\documentclass[oneside, 11pt]{amsbook}
\usepackage[utf8]{inputenc}
\usepackage{url}
\newtheorem{theorem}{Theorem}[section]
\newtheorem{prop}[theorem]{Proposition}
\newtheorem{corly}[theorem]{Corollary}
\newtheorem{lemma}[theorem]{Lemma}
\newtheorem{property}[theorem]{Property}
\theoremstyle{definition}
\newtheorem{definition}[theorem]{Definition}
\theoremstyle{remark}
\newtheorem{remark}[theorem]{Remark}
\newtheorem{example}[theorem]{Example}
\newcommand{\mres}{\mathbin{\vrule height 1.6ex depth 0pt width
0.13ex\vrule height 0.13ex depth 0pt width 1.3ex}}
\newcommand{\D}{\mathcal{D}}
\newcommand{\T}{\mathcal{T}} 
\newcommand{\Pf}{\mathcal{P}} 
\newcommand{\msf}{\mathsf}
\newcommand{\F}{\msf{F}}
\newcommand{\Ss}{\mathcal{S}}
\newcommand{\Hm}{\mathcal{H}}
\newcommand{\RR}{\mathbb{R}}
\newcommand{\CC}{\mathbb{C}}

\newcommand{\Rcl}{\mcl{R}}
\newcommand{\clU}{\mcl{U}}
\newcommand{\mbf}{\mathbf} 
\newcommand{\mcl}{\mathcal}

\newcommand{\ZZ}{\mathbb{Z}}
\newcommand{\NN}{\mathbb{N}}
\newcommand{\Y}{\mathcal{Y}}
\newcommand{\X}{\mathcal{X}}

\newcommand{\wrap}{\curvearrowleft \hspace{-3pt}}
\newcommand{\measRestrict}{\llcorner} 
\newcommand{\nvar}{q} 
\newcommand{\Xdim}{r} 
\newcommand{\blds}{\boldsymbol} 
\newcommand{\WdgeXdim}{{\textstyle \bigwedge} {}_{{}_{\scriptstyle \Xdim}}}
\newcommand{\partlyto}{\dashrightarrow}  
\newcommand{\datum}{3/10/'26}            
\newcommand{\isomto}{\cong}     
\newcommand{\homeomto}{\approx}   
\newcommand{\ess}{1}
\newcommand{\twoess}{2}		
\newcommand{\threeess}{\bar{\mu}}		
\newcommand{\sixess}{2 \bar{\mu}}

\newcommand{\sd}{\text{sd} \,}
\newcommand{\OneNbSpace}{\mbf{P}}

\newcommand{\bb}{\mbf{b}}

\newcommand{\deli}{\delta^{i}}
\renewcommand{\thesection}{\thechapter.\arabic{section}}
\numberwithin{equation}{section}
\numberwithin{theorem}{section}
\numberwithin{figure}{chapter}
\usepackage{amssymb, latexsym, amscd}
\usepackage{epsfig}
\usepackage[mathscr]{eucal}

\begin{document}

\title{Singularity of Data Analytic Operations}
       
\author[S.~P.~Ellis]{Steven P.\ Ellis}
\date{\datum}

\email{spe4ellis@aol.com}


        
\thanks{This research is supported in part by United States PHS grants MH46745, MH60995, and MH62185.}

\thanks{\datum}

\clearpage

\maketitle

  \begin{center}
I dedicate this book to Marjorie, Hollis, and Laura and to my brother Wayne who helped and encouraged me when, as a high school student, I was teaching myself college mathematics. 
  \end{center}

\clearpage

\tableofcontents

\clearpage

\chapter*{Preface} 
This book is aimed at two audiences. The first is a general one: Data scientists and mathematicians interested in data science. The second audience consists of a specific data scientist and mathematician interested in data science: Me. The book is a record of my work I consult and use to verify its correctness. For that reason the book tends to be overly detailed. (Plus, I like details.) Sorry. 

This is not the final draft of this book. There is a small percent of it I know needs more work. 
 
If the reader discovers mistakes, even typographical, in what follows or knows of additional or better literature worth citing -- especially literature devoted to exploring the same issues -- I would be grateful if they sent me an e-mail informing me of this. That is another reason why this is not the last draft of this document. 

Importantly, this book has not been peer reviewed. I'm fairly confident that there are no serious errors in this book but I might be mistaken. I would be grateful if some readers gave parts of it careful scrutiny. 

I guess nowadays one should acknowledge any assistance by AI's in one's writing. Perhaps behind the scenes, unbeknownst to me, my computer and/or software were using AI's but \emph{I have not consciously employ artificial intelligence anywhere in the writing of this document.}

The first draft of this book was posted on the arXiv in July, 2013. This may be the last draft posted on the arXiv. See ``sites.google.com/view/stevenpellis'' for future ones. 

\chapter{Introduction}  \label{Chptr:intro.1}
\section{Intro to the intro}  \label{S:intro.to.intro}
In this book ``data'' means empirical data, the result of counting, classifying, and/or measuring things in the real world. 
``Data analysis'' refers to procedures for extracting useful information or summaries from data (Tukey \cite[p.\ 2]{jwT62.FutureDataAnlys}). 
Descriptive statistics, statistical inference, statistical decision making, and machine learning are all forms of data analysis.  

Empirical data take the form of values of ``variables''. By definition, variables vary. Some of that variation is related to other variables that the data analyst (human or not) knows about and has values for. (Studying the relationship among variables is the main form data analysis takes.) But typically there are many other variables that the analyst has no knowledge of but which influence those recorded in the data. Variation caused by this second group of variables is called ``noise''. The data collection process is usually noisy. 

Empirical data can tell a story about the world. Data analysis is used to reveal that story. But because of noise, if one repeats the process used to produce a data set, the new data set will not be exactly the same as the first. Yet, both data sets should tell essentially the same story. I.e., findings from data analysis should be reproducible. So analyses of the two data sets should produce similar results.
 
This book is about how, typically rarely, the results of a data analysis are not reproducible because of instability inherent in the methods chosen for the analysis. 
Figure \ref{F:HeightFen} illustrates this. It shows lines fitted to two real data sets using ``Least Absolute Deviation (LAD)'' regression. The LAD line is the one that minimizes the sum of the vertical distances from the data points to the line. 
The solid lines in the figure are fitted LAD lines. As described in the figure caption, the dashed lines are the LAD lines for microscopically perturbed data in which the points indicated by the arrows are perturbed in the direction indicated by the arrows. The effect of perturbing the data on the lines is far out of proportion to the microscopic size of the perturbation. One fears that if one repeated the study that yielded these data, the results could be quite different. (As a practical matter, in panel (b) the disturbance to the line is likely too small to be important. We say the apparent \emph{severity} of the disturbance in the panel is mild. We will see that there are many data sets at which LAD is severely unstable.) This book is about instability in data analysis.  

Data analysis uses maps (we call them ``data maps'') that extract structure from noisy data. Call such a map a ``data map''. This book uses mainly topological, geometric, and measure-theoretic methods in a general way to study the behavior of data maps. (In this book we focus on maps used in working with data. However, our results apply to maps quite broadly.) Statistical theory consists mainly of probabilistic treatments of data analysis, but there is little of that in this book. 
 
     \begin{figure}
      \epsfig{file = 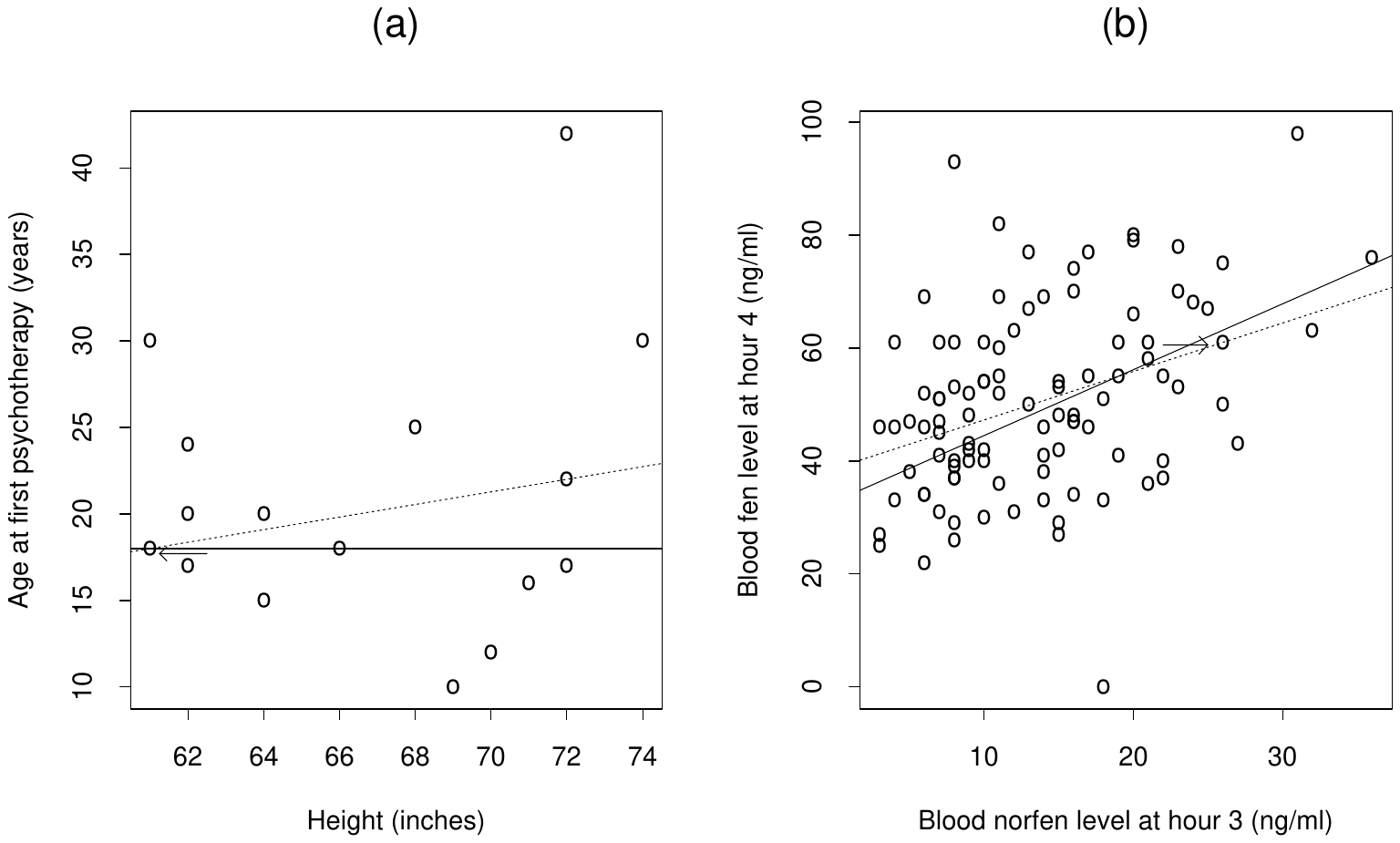, height = 4.0in , , } 
      \caption{Real data sets apparently very near singularities of LAD. Solid lines are LAD lines for data plotted. Dashed lines are LAD lines for data sets obtained by moving observations indicated by arrows a microscopic amount (1/20,000 of the interquartile range of the variables on the $x$-axes) in directions shown (from \cite{spE98}; data courtesy of the Area of Molecular Imaging and Neuropathology, John Mann chief, at the New York State Psychiatric Institute at Columbia University).}    \label{F:HeightFen}
      \end{figure}

As we will see, a likely explanation for the extreme sensitivity of LAD at the data sets in figure \ref{F:HeightFen} is that these data sets are near discontinuities of the LAD map. 
Actually, ``discontinuity'' is not the right term. Continuity of a function at a point depends partly on the value of the function at the point. We will be concerned with the instability of data maps near points at which the map may not even be defined. And even if it is defined there we will not be interested in its value, because one expects the probability of getting a a bad data set to be virtually 0. 

Instead of ``discontinuity'' we use the term ``singularity''. Denote by $\D$ the space of all possible data sets relevant to a given data analytic problem. It might be called the ``sample space''. We call it the ``data space''. Let $\F$ be the ``feature space'', the space of all possible patterns in data in $\D$ that are of interest. Let $\Phi : \D \partlyto \F$ be a data map. (The symbol ``$\partlyto$'' indicates that $\Phi$ might not be defined literally everywhere in $\D$.) A data set $x \in \D$ is a \emph{singularity} of $\Phi$ if $\lim_{x' \to x} \Phi(x')$ does not exist. Here, the limit is taken through a dense subset of $\D$ on which $\Phi$ is defined and continuous. Let $x \in \D$ be a singularity. 

An important point is that, while commonly in practice the probability of getting a singularity as a data set is 0, \emph{the probability of getting a data set \emph{near} one is positive, so getting a data set near $x$ does happen in practice.} (As in figure \ref{F:HeightFen}?) And near a singularity the data map $\Phi$ will be unstable. It rarely happens that one encounters a data set near a singularity of a data map one is using. But it does happen. We come back to this issue in section \ref{SS:singularity}. 
 
It turns out the singularity phenomenon is not a superficial flaw in a data map but depends in a deep way on the topological structure of the question a data map is designed to answer. 

Often (usually?)\ \emph{after} a analyst looks at the data they make a subjective judgment about how to proceed with its analysis. Since the judgment was made after looking at the data, that mysterious subjective judgment \emph{is part of the data map}. Thus, data analysis is often (usually?)\ non-algorithmic. Topology can say something about vaguely defined maps. Example: Fixed point theorems. This means that to the extent topological methods can say something about data analysis, the things it says are realistic. Moreover, singularity is a topological phenomenon. (It is preserved by homeomorphism.) 

In this book we focus on topological sources of singularity that are not immediately obvious, but for starters, consider a situation where it is.

  \begin{example}[College admissions] \label{Ex:college.admissions}
In the United States, two pieces of information colleges often use in deciding whether to admit or reject an applicant are the applicant's scores on the ``Scholasitc Apptitude Test (SAT)", a series of standardized tests, and the student's high school grade point average, their "GPA". For simplicity assume there is just one SAT score. Consider a hypothetical college that bases its entire admit/reject decision on those two numbers.  
This college's decision rule can be portrayed graphically as in Figure \ref{F:college.admissions}. The pairs of scores for which the college will admit the applicant form a region of the product of the possible ranges of GPA and SAT. Similarly for pairs leading to rejection. 

  \begin{figure}
      \epsfig{file = 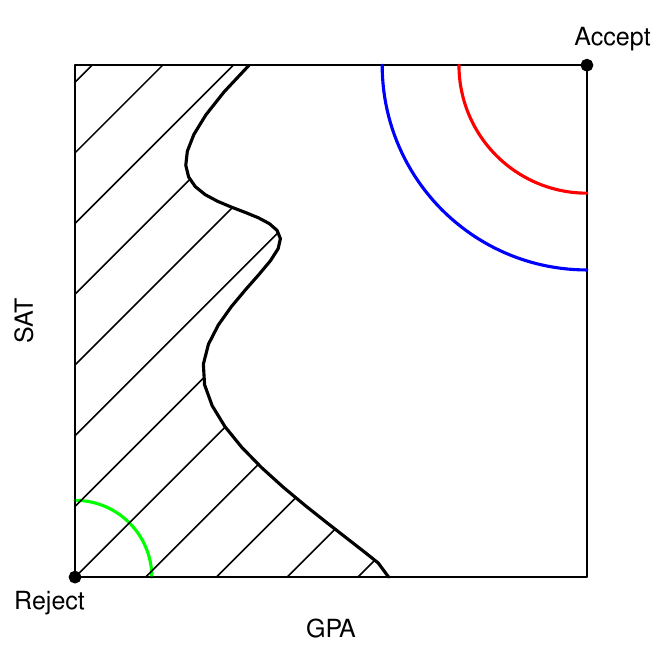, height = 5.5in, , }
      \caption{College admission decision rule. The horizontal axis is high school Grade Point Average. The vertical axis is Scholastic Aptitude Test score. A hypothetical college bases its admissions decision entirely on these two numbers. The lower left and upper right corners are ``perfect fits'' corresponding to rejection and acceptance, respectively. The black wavy curve is a possible boundary between the regions corresponding to rejection (shaded) and acceptance. The red, blue, and green curves are other possible boundaries separating the admission region from the rejection region.}  \label{F:college.admissions}
  \end{figure}

The admissions exercise is pointless if all applicants are accepted or all are rejected. It makes sense that applicants with lowest possible GPA and lowest SAT (lower left corner of the diagram) are rejected and those with highest scores (upper right corner of the diagram) are accepted. 

In data analysis there are usually data sets, in this case $(GPA, SAT)$ pairs, at which data maps of a certain type are \emph{a priori} constrained to take certain values. We call such data sets ``perfect fits'' for that class of data maps. (In fact, these constraints largely define the class of data maps.) So here the lower left and upper right corners of the plot are perfect fits. We denote set of perfect fits by $\Pf$. So in this example 
$\Pf$ consists of two points.

By virtue of the mandated action on the perfect fits, the admissions process for the hypothetical college maps a connected space, the $GPA$--$SAT$ rectangle, onto a discrete one, $\{ACCEPT, REJECT\}$. Such a map cannot be continuous. 

The college is unlikely to be so exclusive (inclusive) that it only accepts (respectively, rejects) applicants with pairs of perfectly good (bad) scores. Moreover, that policy would be extremely sensitive to random variation in the performance of the very best or very worst students. So the acceptance region will include a neighborhood of the upper right corner and exclude a neighborhood of the lower left corner of the diagram. 

In this book, our emphasis is on the \emph{set} of a data map's singularities. We call that set the ``singular set'' of the map. Thus, the boundary of the acceptance region is the singular set of the college's decision rule. Call the singular set $``\Ss$''. In the figure the boundary, singular set, is the black wavy curve.

An important feature of a singular set is its size. One way to describe the size of a singular set is by its dimension. In the figure the singular set is a curve and therefore one-dimensional. (It is often more convenient to use ``codimension'' instead of dimension. The codimension of the singular set is the dimension of the data space minus the dimension of the singular set. So a high-dimensional singular set will have low codimension.) Having established the singular set's dimension, one gains further information from the singular sets's measure. In this case, length.

But length is connected to another important feature of the singular set: Its distance to the set $\Pf$ of perfect fits. We are sure that an applicant with perfect $GPA$ and $SAT$ should be admitted. How, then, could we be unsure about an applicant whose scores are almost perfect? As we recognized before, just by luck an exceptional student might have scores a small distance away from perfect. It would be a pity to reject him/her. On the other hand, it seems reasonable that the admissions office might be unsure about a student whose scores are neither very good nor very bad. In general, singularities, especially severe ones, are less troubling the further they are from $\Pf$. 

Figure \ref{F:college.admissions} shows alternative singular sets corresponding to alternate versions of the decision rule. These are colored red, blue, and green. For each of them the distance of each point on it to (the nearest point of) $\Pf$ is constant. Let $R > 0$ be the constant. Then the length of that singular set is $\pi R/2$. In general, if $R$ is the distance from a singular set to $\Pf$, then the length of that singular set is no smaller than 
$\pi R/2$ We see that as $R$ increases so must that minimum length. So there is a tradeoff between singular set measure (bigger is worse) and distance from $\Pf$ (bigger is good). More complex examples are discussed in section \ref{SS:measure.distance.to.P}. 

A technical problem arises here. In the figure it is easy to see that the singular sets are curves. Their dimension is 1 and length is the appropriate way to measure their volume. But in general, the boundary between the rejection and acceptance region could be very pathological, not at all resembling a curve. In order to discuss this issue in general we use Hausdorff dimension and measure (appendix \ref{Chptr:Lip.Haus.meas.dim}).
 \end{example}

Another aspect of singularity, one we have already mentioned, is its severity. In the college admissions example, all the singularities have the same severity: Near a singularity a tiny change in an applicant's scores can mean the difference between admission and rejection. Figure \ref{F:HeightFen} suggests that in fitting LAD lines, severity of singularity can vary. We will see that severity can be connected to dimension and measure. 

The main theme of this book is what I call the ``Sales Pitch'': One can sometimes get global information about the singular set of a data map by examining its behavior near a  small space, sometimes a very small space, $\T \subset \Pf$, of ``test patterns'' on which the data map behaves in simple way. 

We have already seen an example of this: Knowing that the college admissions process admits students whose scores lie near the upper right corner of figure \ref{F:college.admissions} but rejects students whose scores lie near the lower left corner tells us that the process will have singularities and the singular set will have dimension at least 1. If we look out a distance of $R > 0$ from the perfect fits at diagonally opposite corners and see no singularities, then we know that the length of the singular set $\Ss$ is at least $\pi R/2$. 

In the rest of this introduction, we elaborate on the matters mentioned here. In section \ref{S:preview} pointers are given to places in the book where these issues are discussed in depth.

\section{Instability}  \label{SS:instability}
By ``conditioning'' of an operation we mean the sensitivity of the output to small changes in the input. This book applies topological, geometric, and measure-theoretic methods to study instability or ill-conditioning of data maps. A data map, $\Phi$, is stable or ``well-conditioned'' at a data set $x$ if a small change in $x \in \D$ leads to only a relatively small change in 
$\Phi(x)$. 

If this fails, $\Phi$ is unstable or ``ill-conditioned" at $x$: Small changes in $x$ can lead to relatively large changes in the output. As explained in section \ref{S:intro.to.intro}, instability at $x$ renders $\Phi$ useless in the vicinity of $x$.

Instability is a basic issue in applied
mathematics (e.g., Hadamard \cite[p.\ 38]{jH23}, Isaacson and Keller \cite[p.\ 22]{Isaacson.and.Keller}, and Tikhonov and Arsenin \cite[pp.\ 7--8]{anTvyA77.illposed}). Our work is especially relevant to numerical analysis (Higham \cite{njH02.StabNumrAlg}).
The following are also of interest:  
Blum \emph{et al} \cite[Chapters 11 and 12]{lBfCmSsS98.realcompute}, Demmel \cite{jwD88.numanal}, and Beltr\'{a}n and Pardo \cite{cBlmP07.ConditionNum}. 

Various forms of stability of data maps are discussed in 
Poggio and Smale \cite[p.\ 543]{tPsS03.learn}, Mukherjee \emph{et al} \cite{sMpNtPrR03.stability}, 
Breiman \cite{lB96.BagPrdctrs,lB96.instablty}, 
Berhane and Srinivasan \cite{iBcS04.StabltyNN}, Yu \cite{bY13.stability}; Rinaldo \emph{et al} \cite{aRaSrNlW2012.ClusteringStability};  
see also Obenchain \cite[Lemma 1, p.\ 1571]{rlO71.LinTransInvar}). 
Dreossi \emph{et al} \cite{tDsGaS-VsaS2019.RobustnessOfNets} study instability in deep neural networks.
B\"{u}hlman and Yu \cite{pBbY02.Bagging} formalize the notion of stability asymptotically in the sample size. We do not rely on such asymptotics in this book.  

\section{Singularity} \label{SS:singularity}
In the preceding section we said $\Phi$ is unstable at $x \in \D$ if small changes in $x$ can lead to relatively large changes in $\Phi$. In this book we are interested in data sets $x$ with the property that \emph{arbitrarily} small change in $x$ can lead to \emph{arbitrarily} large relative changes in $\Phi(x)$. We call such a data set a ``singularity'' of $\Phi$. 

To make that more precise, let $\D' \subset \D$ be dense in $\D$ and suppose $\Phi$ is defined and continuous on $\D'$. A data set $x \in \D$ is a ``singularity'' of $\Phi$ (with respect to $\D'$) if the limit $\lim_{x' \to x, \, x' \in \D'} \Phi(x')$ does not exist. If $\F$ is a complete metric space then $x$ is a singularity in the sense of the last paragraph. Whether $x$ is a singularity of $\Phi$ has nothing to do with $\Phi(x)$ or even if $\Phi(x)$ is defined. Practically speaking that does not matter. The chance of getting $x$ as a data set is likely 0 anyway.
 
It turns out that singularity is a deep-seated phenomenon in data analysis, making \emph{data maps} with singularities quite common. (However, for a given data map with singularities, one expects \emph{data sets} near its singularities to be rare. See section \ref{SS:examples} for a peek at the examples we look at in this book.) Our basic goal is to come up with easy to check (or at least not too hard to check) conditions under which functions -- we have in mind data maps -- have singular sets of at least a certain size and, possibly, of a certain level of severity.

Call the set of singularities of a data map its ``singular set''. A measure of the conditioning of a data map is the size of its singular set and the severity of the singularities therein. The reason singularity is important in data analysis is that \emph{data maps are unstable \emph{near} their singular sets.}   

  \begin{example}[Hypothesis testing]  \label{Ex:hypothesis.testing}
(Lehmann \cite{elL93.StatHyps}) This is the main statistical method
used in biomedical research. The ``feature space'', $\F$, consists of two points, \linebreak
$\{ACCEPT, REJECT\}$ or $\{0,1\}$. An infamous instance of singularity occurs in ``fixed level testing'', defined as follows. One has a function 
$p : \D \to [0,1]$. (The values of $p$ are called ``$p$-values''.) For convenience assume $p$ is defined everywhere on $\D$. 
Define $\Phi(x) = ACCEPT$ if $p(x) \geq 0.05$ and $\Phi(x) = REJECT$ 
if $p(x) < 0.05$. In the latter case one says that the finding is ``statistically significant (at the $\alpha = 0.05$ level)''. (The 0.05 is standard, but sometimes other ``alpha levels'' are used.) When the data space is connected the singular set is $p^{-1}(0.05)$. Its codimension needn't be greater than 1. 

Unfortunately, many medical and other journals will not publish a finding that is not statistically significant. This is a real problem in medical research and there is a large literature inveighing against the use of ``fixed level testing'' (Wasserstein \emph{et al} \cite{rlWalSnaL2019.pValues}). 

Various alternative ways of assessing strength of evidence in data have been put forward. However, the problem is inescapable because data are often used to make decisions from a discrete set of alternatives. For instance, the U.S. Federal Drug Administration (FDA) uses data, usually company provided, to decide whether to approve a drug or medical device for use in medicine. That is a binary decision, approve or disapprove. But the relevant data space is usually a connected set and the FDA is required to actually approve some drugs and
disapprove others, so topology forces there to be borderline cases
of data that are singularities of the decision process.  
  \end{example}
  
Our view in this book is that singularity is deleterious. Singularity is not an absolute evil, however. One might be willing to accept a large singular set in exchange for improvement in some other aspect of data map performance. Sometimes by allowing the singular set to be larger one can reduce the ``severity'' of the singularities. (See \cite[Theorem 2.5]{spE91.sings.plane.fit}. "Severe" has a different meaning there than in this book.) A statistical method should not be rejected just because it has a large singular set.

However, given two data maps of the same sort, \emph{everything else being equal}, one prefers the one with a small singular set. More precisely, there are aspects that one would like to be equal and while there are others one does not care about. See section \ref{SS:measure.distance.to.P}
for discussion of a data map attribute that one might want to be equal in order that the singular sets be considered comparable.

For frequently used methods how often does one get a data set near a singularity?  
In practical statistical work, it is uncommon but not terribly rare for statistical software to fail and report that some matrix is nearly singular or that an iterative procedure did not converge. Might such data be near a singularity of the statistical method being used? In such a case, the statistician has little choice but to use a different statistical method. Switching statistical methods in that fashion, however, can itself create singularity. In such a case it makes sense to report the results of the alternative analysis but mention the failure of the original method as a kind of ``diagnostic''.  

It is safe to assume that singularities cause trouble only a small fraction of the time (less than 1\%?). In section \ref{SS:dimension}, the probability of getting data near the singular set is analyzed quantitatively in a specific toy example. 
However, data analysis is not crippled by singularity. Statistical methods work! (Most of the time.)

Some data analyses are critical. Their failure can have serious consequences. Important and expensive policy decisions may be influenced by spurious results. A spurious result due to singularity may inspire useless further research. Autonomous operation of a car or truck or of a space probe may fail. Might the algorithms of financial engineering be inherently subject to singularity?

Suppose a data map $\Phi$ estimates some quantity of interest. Consider a data set $x$ close to a singularity of $\Phi$, but such that (s.t.)\ $\Phi(x)$ is defined. The estimates computed may be unreliable. One would like this unreliability to be reflected by the standard errors of the estimate. But, typically, standard errors are just approximations whose validity rests on so called ``asymptotic'' arguments, ones describing statistical behavior in the limit as the amount of data approaches infinity. But such asymptotic approximations may not work well for data sets of finite size close to a singularity. (This kind of asymptotics are not employed in this book.)

The risk of getting data near a singularity is real and the consequences can be serious. Unfortunately, it appears to be hard to remove this danger or to recognize its presence. 

   \begin{remark}[``Sales Pitch'']  \label{R:sales.pitch}
Singular sets can be complicated. In this book we show how one can sometimes get global information about the singular set of a data map by examining the map's behavior near a small space, $\T$, of ``test patterns'', on which the data map behaves in simple way. I.e.,  global conclusions are drawn from local information. That is the argument, ``the sales pitch'', for using the methods described here.
  \end{remark}

  \begin{remark}[Learning and predicting]  \label{R:learning.and.predicting}
A data operation, $\mcl{L}$, ``learns'' from a data set $x \in \D$ (when possible) an object, $\mcl{L}(x)$, in a space $\mathrm{F}$. We say that $\mcl{L}(x)$ is ``trained'' on $x$. 

Sometimes the objects in $\mathrm{F}$ are themselves maps. Let $\Y$ be the common domain and $\mcl{Z}$ the codomain of those maps.  
Write $f := \mcl{L}(x)$. Given an input $y \in \Y$, $f(y) \in \mcl{Z}$ is a ``prediction''. For example, $\Y$ might be $\RR$ with its points interpreted as SAT scores (example \ref{Ex:college.admissions}) and the points of $\mcl{Z} = \RR$ may have the interpretation as students' grade point averages (GPAs) at the end of freshman year of college (not high school GPA as in example \ref{Ex:college.admissions}). If $\D$ is the space each of whose points is a finite collection of pairs 
$(SAT \, score, freshman \, GPA)$ then by applying some regression method to a point $x$ 
in $\D$ we ``learn'' a function $f = \mcl{L}(x) : \Y \partlyto \mcl{Z}$ that can be used to compute from a student's SAT score a prediction of his/her freshman GPA. 
 
Other possible examples, besides those discussed in this book, are when the points 
of $\mcl{Z}$ are complicated robot motions (Farber \cite[Chapter 4]{mF08.TopolRobot}) or self-driving car maneuvers.

We can interpret $\Y$ as a kind of data space and the theory described in this book may have something useful to say, not just about $\mcl{L}$ (learning), but also about $f$ (predicting).

To simplify the analysis the learning operation, $\mcl{L}$, we might extract from the maps 
in $\mathrm{F}$ some geometric feature (graph, for example) and derive a data map 
$\Phi : \D \partlyto \F$, where $\F$ is the space of the derived geometric features.
  \end{remark} 

\section{Calibration}   \label{SS:calibration}  
Our method is based on what might be called ``calibration:''  A data summarization method designed to detect a certain kind of structure in data must find that structure, at least approximately, when it is present in pure, perfect, or at least strong form.
Data sets having the structure in pure, perfect, or at least strong form we will call 
``perfect fits.'' A data set is a perfect fit if there is no disputing if and how it manifests the structure of interest. For example, in fitting $k$-planes to data 
if the data lie exactly on a unique $k$-plane there is a canonical choice of plane to fit to the data: The plane on which they lie exactly. Any reasonable plane-fitting method should fit the right plane to (almost all) such data sets, at least approximately. 

Another example is division into two clusters in batches of numbers. If two subsets of a batch are separated by six times the range of values in either subset, then it is quite clear what the two clusters are. It is even clearer is the subsets are separated by 100 times the range of values in either subset, but such an extreme separation is not necessary in order to have a ``perfect fit'' for the binary clustering problem. 

Let $\Pf \subset \D$ denote the set of ``perfect fits''. $\Pf$ is the ``perfect fit space'' for the data analysis of interest. If $x \in \Pf$ then there is little if any choice in $\Phi(x)$.  
If the data map finds the correct structure at almost all data sets in 
$\Pf$, we say that it is ``calibrated'' with respect to (w.r.t.)\ $\Pf$. (Often finding approximately the correct structure sufficies.) Calibration breaks down in interesting ways when a data summarization method is regularized in order to improve its generalization properties.
 
$\Phi$ must (usually) exhibit (approximately) correct behavior on $\Pf$. A ``standard'', 
$\Sigma: \Pf \to \F$,  defines what that correct behavior is. $\Sigma$ is a rule that 
$\Phi$ should approximate on almost all of $\Pf$. One might think of a  class of data maps (or a data analytic problem) as specified 
by the quadruple $(\D, \F, \Pf, \Sigma)$.\footnote{Here we are flirting with a category-theoretic formulation (Riehl \cite{eR2014.CatThy}) of data analysis. But what are the morphisms? What are the functors?}
One can think of $\Sigma^{-1}$ as a set-valued partial specification of a ``forward problem'', $\F \to \D$. 

Often we can get important information about a data map by studying its behavior on a small subset $\T \subset \Pf$. We call $\T$ the ``test pattern space''. Sometimes $\T$ is much smaller than $\Pf$.  
This is the ``sales pitch'' (remark \ref{R:sales.pitch}).

\section{``Line-fitting''} \label{SSS:LF.plots}   
Example \ref{Ex:college.admissions} is an obvious example of how a data map, for topological reasons, may have singularities. In this section we discuss a more subtle example. 
 
We call a finite collections of points a ``point cloud''. As mentioned in section \ref{SS:calibration}, a plane-fitter is a plane-valued data map that takes a point cloud as input and assigns to it a plane, of dimension $k$, say, in such a way that if the point cloud lies exactly on a unique $k$-plane then the plane on which the cloud lies exactly is the one assigned to it. (A null set of exceptions are permitted.) 

Studying a function via its graph is a common mathematical tactic. To understand a plane-fitter better might we draw its graph? This is hard. For one thing, the spaces involved are high dimensional. So settle for ``line-fitting'': Fitting a line (1-plane) to a point cloud on the plane. Figure \ref{F:HeightFen} is an example. To simplify further consider lowest dimensional non-trivial case: Fitting a line to 3 points on a plane. The dimension of domain is 6, still high  dimensional. We will restrict attention to a 2-dimensional subspace of $\RR^{6}$. 

The line fitted to a point cloud doesn't have to pass through the origin. To simplify further just take the line through the origin parallel to the fitted line. Then the codomain is the ``real projective line'', $P^{1}$. It is homeomorphic to a circle, as such a little tricky to plot on the $y$-axis. 

After all these simplifications we can make (partial) graphs (``LF'' plots, \cite{spE.3.or.4}) of a line-fitter. Let $\Delta \subset \RR^{2}$ be the triangle (2-simplex) whose vertices are 
$(1,0)$, $(0,1)$, and $(-1,0)$. Then $Y_{1} = (x_{1}, y_{1}) \in \Delta$ 
if and only if  $y_{1} \geq 0$ and $|x_{1}| + y_{1} \leq 1$. 
Let $Y := \xi(Y_{1}) \in \D$ be the data set whose first row is just $Y_{1}$, whose second row is $Y_{2} := \bigl( 1-|x_{1}|-y_{1}, 0 \bigr)$, and whose third row is $Y_{3} := -Y_{1}-Y_{2}$. So each point of $\Delta$ corresponds to a point cloud on $\RR^{2}$. Define $\X$ to be the image of $\Delta$ under $\xi$. $\X$ is obtained by taking $\Delta$, folding it along its vertical midline, and imbedding it into $\RR^{6}$. 

One makes an LF plot of $\Phi$ as follows. Through each point $Y_{1}$ on a grid 
in $\Delta$ draw a short line segment parallel to the line $\Phi$ fits to $\xi(Y_{1})$. Figure \ref{F:LF.plot.pipeline} shows the procedure graphically.

       \begin{figure}
             \includegraphics[width=6.2in, height = 4.7in, , angle = 0]{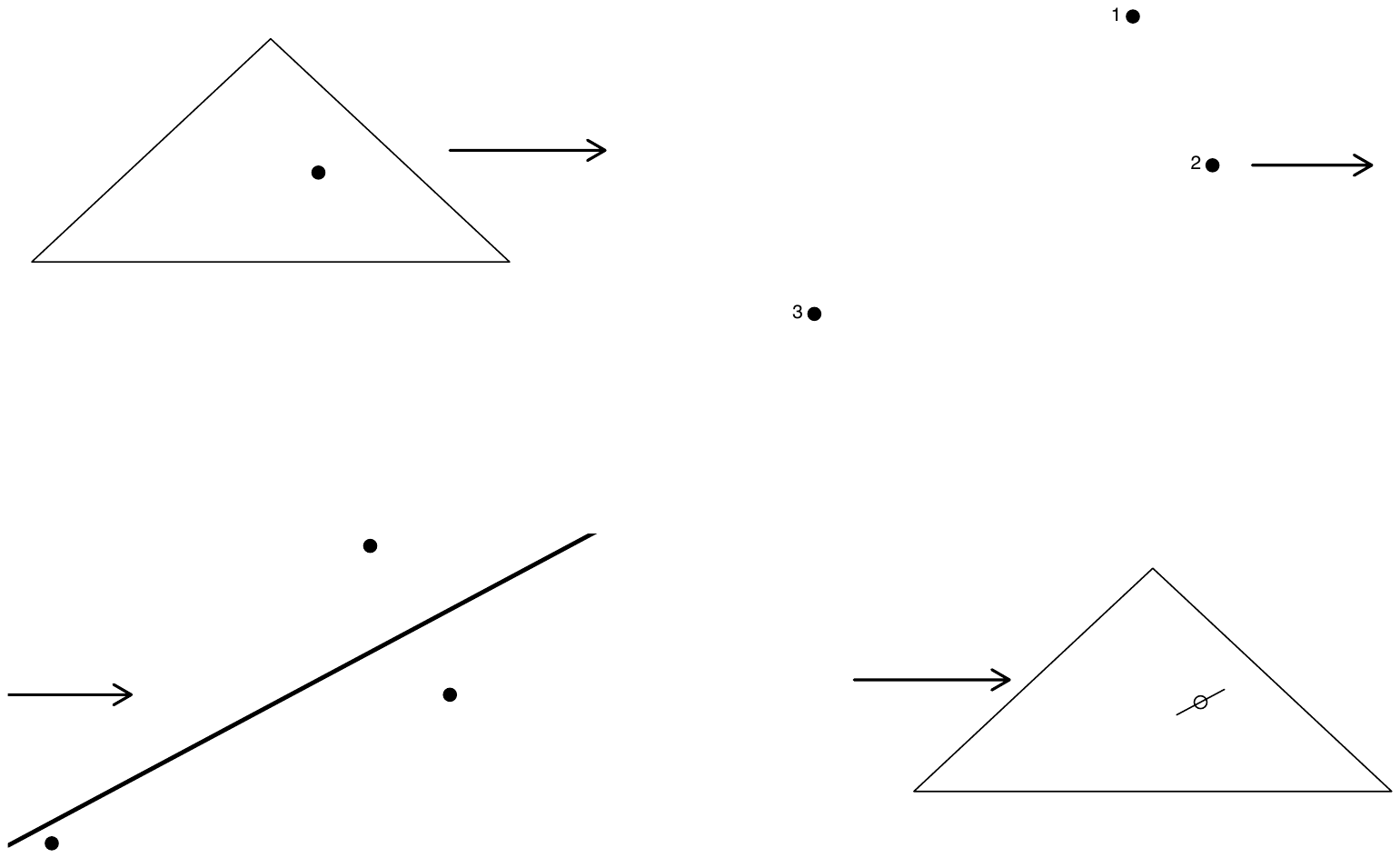}
          \caption{The procedure for making an ``LF'' plot for a line-fitting method. Upper left: Pick a point in the triangle $\Delta$. Upper right: That point encodes a data set consisting of three points on the plane. Lower left: If possible use the line-fitting method in question to fit a line to that data set. Lower right: Through the point in 
$\Delta$ that one started with draw a short line segment parallel to the fitted line. Repeat this process for each point in a grid in $\Delta$.} 
\label{F:LF.plot.pipeline}
       \end{figure}

Figure \ref{F:LS.PC.LAD.lf.plots} shows LF plots for three commonly used plane-fitting (in this case, line-fitting) methods. We have already mentioned LAD. Least Squares (linear regression) fits the line that minimizes the sum of the \emph{squares} of the \emph{vertical} distances from the points in the cloud to the line. Principal components line fitting (PC) finds the line that minimizes the sum of the squared \emph{perpendicular} distances from points in the cloud to the line.
In each case, LAD, LS, and PC, there is at least one data set at which the line-fitting method is unstable.

       \begin{figure}
             \includegraphics[width=6.2in, height = 4.7in, , angle = 0]{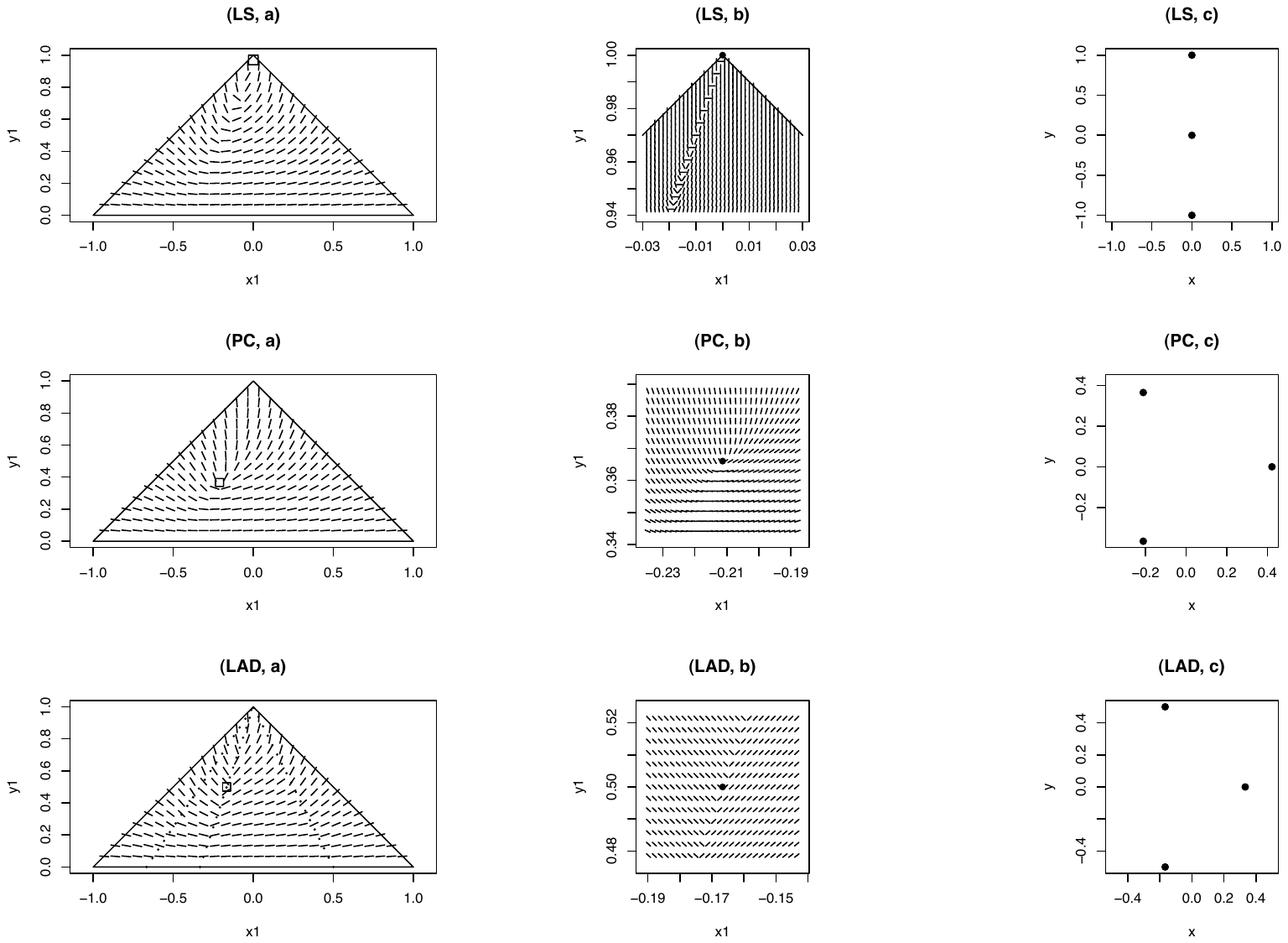}
          \caption{Fitting lines to samples of three bivariate data points. ``(a)'' panels:  LF plots for least squares (LS), principal componensts (PC), and least absolute deviation regression ($L^1$ or LAD). Small rectangles enclose singularities. (All points on dashed lines in (LAD,a), except the endpoints, are singularities of LAD.) 
``(b)'' panels:  Blow-up of rectangles in ``(a)'' panels. Singularities are
indicated by dots. ``(c)'' panels:  Scatterplots of data sets indicated  by dots in ``(b)'' panels (adapted from Ellis \cite{spE.3.or.4}).} 
\label{F:LS.PC.LAD.lf.plots}
       \end{figure}
Figure \ref{F:bound.cond.and.whorls} explains why we see instability in figure \ref{F:LS.PC.LAD.lf.plots}. Points on boundary of $\Delta$ in an LF plot correspond to ``perfect fits'': The data sets corresponding to boundary points each consists of three points lying exactly on a unique line. By definition of line-fitting, a line-fitter assigns that line to the data set. (So the perfect fits serve to ``calibrate'' the line fitter; section \ref{SS:calibration}.) 

Thus, defining a line-fitter amounts to solving a ``boundary value problem":  Try to extend the line fitter continuously over interior of the triangle. Panels (b,c,d) in figure \ref{F:bound.cond.and.whorls} show an attempt at a solution to this boundary value problem. (This is an example of the ``extension problem'' in topology, Spanier \cite[p.\ 20]{ehS66}. Similar ideas can be found in Brezis \cite[section 3, p.\ 191]{hB03.AnlysTopol}.) As we go around the triangle once, the perfect fit line goes around the projective line, $P^{1}$, exactly once as well. Thus, any line-fitter maps the boundary of $\Delta$ onto a generator of the fundamental group or 1-dimensional homology class of $P^{1}$. 
Since $\Delta$ is acyclic but $P^{1}$ is not, this means that no line-fitter can be extended continuously over the whole of $\Delta$. 
This simple example shows that algebraic topology has something to say about data analysis.\footnote{It is easy to see that a data set consisting of three copies of the same point on the plane has to be a singularity of a line-fitter. But it is also easy to see that no point in $\Delta$ encodes such an uninteresting data set. For example, none of the singularities displayed in the (c) panels in figure \ref{F:LS.PC.LAD.lf.plots} have that form.} 

The boundary of the triangle is  the ``test pattern space'', $\T$, for this problem. (See section \ref{SS:calibration}.)   

\begin{figure}
      \epsfig{file = 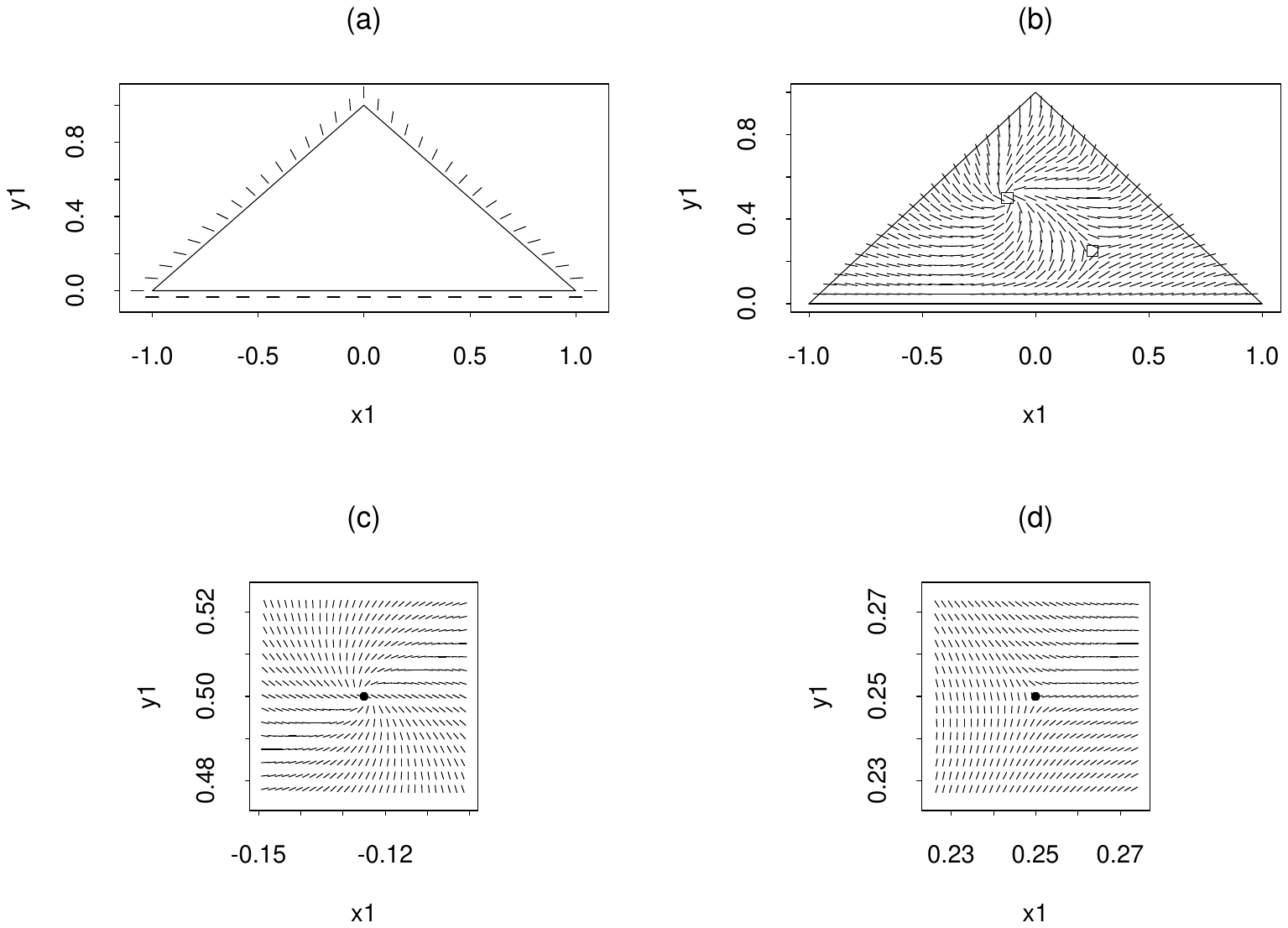, height = 3.6in, , }  
      \caption{Line-fitting is unstable. (a): Boundary conditions to be satisfied by a 
line-fitting method. The boundary of the triangle corresponds to the test pattern space, ``$\T$''. (b):  One attempt at solution to boundary value problem. Small squares enclose regions of instability. (c) and (d):  Blow-ups of squares in (b). Dots indicate data sets at which instability is infinite, i.e., singularities. (from Ellis \cite{spE.3.or.4})}    \label{F:bound.cond.and.whorls}
\end{figure} 

 \section{Examples}  \label{SS:examples}
Here are examples of classes of data maps, all but one analyzed in this book, for which the topology of the problem has implications for stability. 

\begin{enumerate}
  \item \emph{Hypothesis testing}   
  (example \ref{Ex:hypothesis.testing} above). 
     \label{Exmpl:hyp.testing} 
  \item \emph{Plane-fitting} (\cite{spE91.sings.plane.fit,spE95,spE96,spE98,spE00,spE.3.or.4}). 
We touched on this in sections \ref{SS:calibration} and \ref{SSS:LF.plots}. Plane-fitters are data maps that assign planes to point clouds in Euclidean space. Examples include linear regression and principal components plane-fitting (PC). In our analysis we ignore the offset of the fitted plane from the origin, if any, and take the feature space, $\F$, to be a Grassmann manifold. We analyzed the simplest case in the preceding section. 

The test pattern space, $\T$ is homeomorphic to a circle. Each data set in $\T$ lies exactly on a unique plane of the appropriate dimension. For example, in 
section \ref{SSS:LF.plots}, $\T \subset \RR^{6}$ is the image under $\xi$ of the boundary of the triangle used in making LF plots. \emph{Prima facie} the data space is Euclidean. However, it is more convenient to restrict attention to data sets lying on a sphere in that Euclidean space or the one-point compactification of Euclidean space.
  
An important  special case of plane-fitting is linear regression.  
Suppose the data consist of $n$ points in $\RR^{\nvar}$. (Assume $n > \nvar > 1$.) Each point has the form $(\mbf{x},y)$, 
where $\mbf{x} \in \RR^{\nvar-1}$ and $y \in \RR$. Then linear regression at a data set $x$ using a given regression method produces a function, 
$f( \cdot; x) : \RR^{\nvar-1} \to \RR$, whose graph is a $(\nvar-1)$-dimensional affine subspace of $\RR^{\nvar}$. Let $\Phi(x)$ be the plane through the origin parallel to that affine subspace. Figure \ref{F:HeightFen} shows real data sets that are apparently very close to singular set of least absolute deviation (LAD) regression.
      
One apparent difficulty in applying the theory to linear regression is that in practice one is interested in the stability of $f( \cdot; x)$ as an element of a function space, not in the stability of $\Phi(x)$, an element of a Grassmann manifold. However, we will see
that any singularity of $\Phi$ is also a singularity of $f( \cdot; x)$ (but not necessarily conversely). So our geometric point of view is actually a conservative way to study the stability of $f( \cdot; x)$. 
   \label{Exmpl:plane.fitting}
    \item \emph{Location problem on spheres} \cite{spE91.top.direct.axis}.The most basic question one can ask about a point cloud is, where is it? In Statistics this is called the ``location problem''. Data maps that offer answers to that question are called ``measures of location''. The arithmetic mean or median are measures of location for point clouds on the line. The mean and median are continuous, in fact Lipschitz, functions of the data and as such have no singularities. However, measures of location on non-Euclidean spaces typically have singularities (Eckmann \emph{et al} \cite{bEtGpjH62.genMeans}).   
    \label{Exmpl:location.on.spheres}
   \item  \emph{Linear classification} Given a point cloud whose points have labels 
 $\pm 1$ find a plane that  largely separates positive from negative points
   \label{Exmpl:lin.classification}
   \item \emph{Factor analysis} (\cite{spE.fact.anal}). 
Principal components is often lumped together with this unsupervised learning method. 
Factor analysis is a plane-fitting method with some additional structure (Johnson and Wichern \cite[Chapter 9]{raJdwW92}). As such, it has the singularity issues of a plane-fitter. However, for topological reasons the additional structure creates additional singularities . These singularities seem to lie outside the theory developed in this book. I do not pursue this topic here.
\label{Exmpl:fact.anlys}
\end{enumerate}

In this book, in addition to the \emph{classes} of examples \ref{Exmpl:hyp.testing}, \ref{Exmpl:plane.fitting},  \ref{Exmpl:location.on.spheres}, and \ref{Exmpl:lin.classification}, we also examine specific \emph{examples} of these classes in depth.   

\section{Deep neural nets}  \label{S:DNN}
An intriguing and important example we have not examined is deep neural nets (DNNs), an architecture often used in artificial intelligence. Dreossi \emph{et al} \cite{tDsGaS-VsaS2019.RobustnessOfNets} examine stability of trained DNNs. Perhaps the methods described in this book might shed some light on DNN stability. 

Let $\Phi$ be a trained deep neural net. Conventional wisdom holds that $\Phi$ is a ``black box''. (See Knight \cite{wK2017.AIdarkSecret}.) I.e., it is impossible to say how 
$\Phi$ produces output. However, recent work (see, e.g., Ip \cite{shrI2024.DNNexplainability}, Templeton \emph{et al} \cite{aT2024.Interpretability}) has been able to reveal a little about how DNN's work. 
So a little of $\Phi$'s inner workings might be understood. 
This is encouraging because the ``Sales Pitch'' (remark \ref{R:sales.pitch}) asserts that a little information about $\Phi$, \emph{viz.} $\Phi$'s behavior near a relatively low dimensional space of simple inputs, might be enough to reveal information about its singular set.

On a deeper level, let $\Phi$ be the training process of a DNN: $\Phi$ takes training data, usually a huge amount, as input and returns a trained net. (``We do not understand how robust or fragile models are to perturbations to input data distributions,'' Wing \cite{jmW.2020.TenResearchChallengesInDataScience}.) Maybe the theory in this book can be used to study the sensitivity of DNN's to the distribution of the training data.

\section{Topology} \label{S:topology}
Long a branch of pure mathematics, topology has developed an applied side. Physicists have used topology to prove that certain structures or materials must have dislocations (Chen \emph{et al} \cite{bg-gCgpArdK09.SingsInSmectics}, Smalyukh and Lavrentovich \cite{iiSodL06.LiquidCrystals}), and to study ``optical vortex knots'' (Dennis \emph{et al} \cite{mrDrpKbJkOHmjP10.OpticalKnots}).
Topological singularity is also an issue in robotics (Farber \cite[Chapter 4]{mF08.TopolRobot}) and control theory in engineering (Jonckheere  \cite{eaJ97.RobustStability}).
Ghrist \cite{rG14.ElemAppliedTopol} surveys applicable topology. Topology also has applications to social choice theory (Chichilnisky  \cite{gC79.SocialChoiceTopol} and Ghrist \cite[Section 8.7, pp.\ 170--171]{rG14.ElemAppliedTopol}.) Topology has also found application in economics (Wei \emph{et al} \cite{yWjWzW2025.TopolInEcon}).
 
Recently there has been work on the application of algebraic topology directly to \emph{data} (e.g., Niyogi \emph{et al} \cite{pNsSsW11.MoreManifHomolLearn}, 
Carlsson \cite{gC09.TopologyAndData}, 
Adler \emph{et al} \cite{rjAoBmsBeSsW10.PersHomolRndmFldsCmplxs}, 
Bubenik \emph{et al} \cite{Bube:Carl:Kim:Luo:stat:2010}, Edelsbrunner and Harer \cite{hEjlH10.CompTopol}, Ellis and Klein \cite{spEaK14.ConcurTopolfMRI}). And work has been work done applying topology to machine learning (Hensel \emph{et al}\cite{fHmMbR2020.TopolMethodsML}).

The present book, on the other hand, involves the application of algebraic topology directly to \emph{methods} of multivariate data analysis, for understanding their instability, if any. Some of the material appears in the papers 
\cite{spE91.sings.plane.fit,spE91.top.direct.axis,spE95,spE96,spE.fact.anal}. Since the aspects of a  problem that lead  to instability are rather general, I conjecture that many multivariate data analytic methods can fail catastrophically (but, mercifully, probably rarely do). In this book we develop some general theory concerning singularity then apply it to three classes of statistical methods where it occurs in a nontrivial way.

Instability can arise when the feature space, $\F$, has nontrivial topological structure, specifically nontrivial homology. One such $\F$, the projective line, came up in section \ref{SSS:LF.plots}. The projective line is the simplest Grassmann manifold (section \ref{SS:examples}). Maps into such an $\F$ are necessarily nonlinear. 
Statisticians usually analyze nonlinear statistical methods asymptotically (i.e., as the size of a data set goes to infinity). Roughly speaking this amounts to analyzing methods locally. It seems like that the results of such asymptotic analysis will breakdown near a singularity. But many methods of multivariate analysis are nonlinear and by using topological techniques one can analyze nonlinear methods \emph{globally} for fixed, finite sample sizes.

As discussed in section \ref{S:intro.to.intro}, quite often in data analysis or training deep neural nets humans tinker with or moderate the process. This introduces subjective judgment into the data analysis process. As pointed out in section \ref{S:intro.to.intro} that means that in practice \emph{data analysis is not algorithmic.} But topological methods can tell us things even about non-algorithmic maps. As long as the process behaves sensibly near perfect fits, the results presented here may be applicable. 

In addition, singularity, the topic of this book, is a topological phenomenon: It is preserved by homeomorphism.

The focus of this book is statistical data analysis, however, another example of summarizing and learning from data is animal (e.g., human) cognition. Topological laws of data analysis might apply to that (Zeeman \cite{ecZ65.ZeemanTopolBrain}; Ellis \cite{spE.ambiguity.neuro.compute}).

\section{Dimension}  \label{SS:dimension} 
How big are singular sets? A weak answer to this question is given in Ellis \cite{spE96}. A better, easier to work with way to measure ``how big'' is (Hausdorff) dimension, which we denote by ``$\dim$''. 

The co-dimension of $\Ss$ ($:= \dim \D - \dim \Ss$) is related to how fast the probability of being within $\epsilon$ of $\Ss$ decreases as $\epsilon \downarrow 0$. Specifically, 
the probability of getting a data set within $\epsilon$ of a singular set is bounded below by a multiple of $\epsilon^{\text{codim} \, \Ss}$. 
Thus,  $\text{codim} \, \Ss$ is connected to the behavior of left tail of cumulative distribution function (CDF) of distance from randomly chosen data set to $\Ss$. Formally, 
at $\epsilon \geq 0$, the CDF of the random distance to $\Ss$ is the probability that the distance is no greater than $\epsilon$. 

Data maps are unstable near their singularities. The codimension of the singular set is related to the probability of a data set being near to one. Thus, from a practical standpoint, codimension of its singular set is a useful thing to know about a data map. 

\begin{figure}
   \epsfig{file = 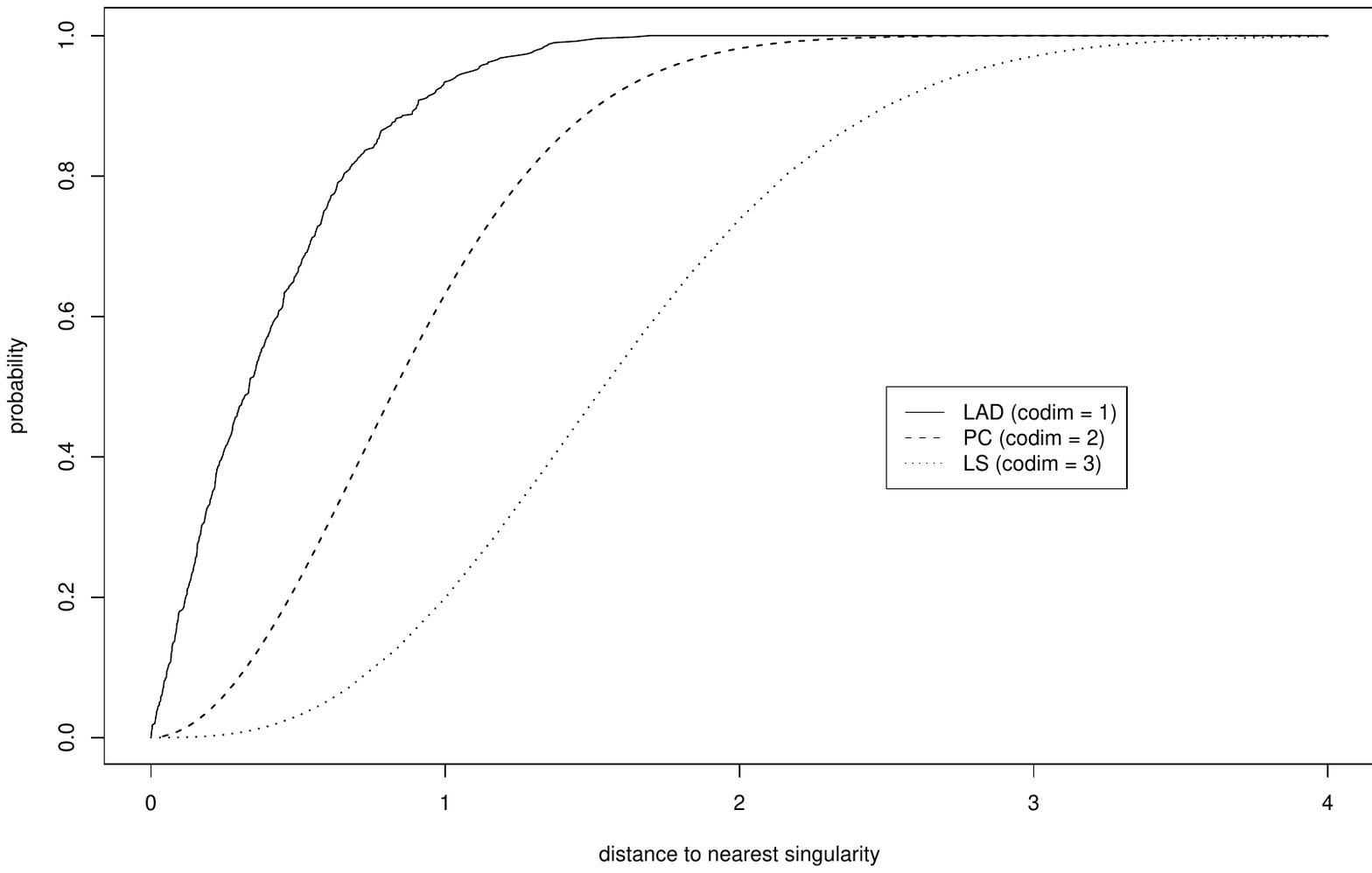, height = 3.5in, , } 
   \caption{CDF's of distances from a Gaussian random data set of four points on the plane to the singular sets of least absolute deviation linear regression (LAD), principal components line fitting (PC), and least square (LS) linear regression in fitting a line to four points on a plane. (From \cite[section 8]{spE.3.or.4})}   \label{F:sing.dists.cdfs}
\end{figure}

In section \ref{SSS:LF.plots} we examined fitting a line to three points on the plane. Richer structure emerges if we instead fit a line to \emph{four} points. Four points in the plane are specified by $2 \times 4 = 8$ coordinates. Consider the random data set generated when the 8 coordinates are independent normal random variables with mean 0 and variance 1. 
In figure \ref{F:sing.dists.cdfs} we plot the CDF of distance to $\Ss$ for fitting a line to the random four points for our usual set of three line-fitting methods: Least Absolute Deviation regression ($\text{codim} \, \Ss = 1$), Principal Component line fitting ($\text{codim} \, \Ss = 2$), and Least Squares regression ($\text{codim} \, \Ss = 3$). We observe that, asymptotically as $\epsilon \downarrow 0$, the CDFs decrease linearly, quadratically, and cubically respectively. 

\section{Measure of singular set and its distance to $\Pf$} \label{SS:measure.distance.to.P}
Hausdorff dimension is a coarse way to measure the size of a set $\Ss$. In this book 
we also derive a lower bound on the Hausdorff measure (appendix \ref{Chptr:Lip.Haus.meas.dim}) of the singular set The lower bound depends on the distance from $\Ss$ to $\Pf$.

In section \ref{SS:singularity}, we stated that there are other considerations that need to be balanced with the size of the singular set in judging a data map. Here we bring up one of them. Start with figure \ref{F:CircleSing}. It has to do with the problem of determining the ``location'' of points on a circle (``directional data'').
Given a data set consisting of a finite collection of points on a circle, a measure of location is analogous to the arithmetic mean or median of a batch of numbers: A single point which gives the location of a point cloud, in this case on a circle. For this problem, the set $\Pf$ of perfect fits is the diagonal consisting of data sets of the form 
$\{ x, x, \ldots, x \}$.

Figure \ref{F:CircleSing} illustrates two measures of location, $\Phi_{1}$ and 
$\Phi_{2}$. The small circles on the large circles in panels (a) and (b) of the figure represent observations in directional data sets. (Statisticians use ``observation'' to mean a single data point.) Panel (a) of the figure shows a singularity of $\Phi_{1}$. (b) shows a singularity of $\Phi_{2}$. 
Surprisingly, the appropriate Hausdorff measure 
of the singular set of $\Phi_{1}$ is apparently 6.6 billion times that 
of $\Phi_{2}$! 

 \begin{figure}
      \epsfig{file = 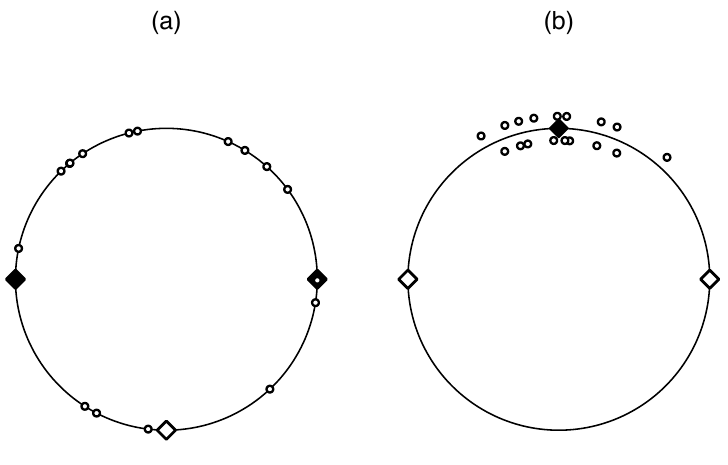, height = 3.7in, , }  
      \caption{Each of the large circles is a unit circle centered at the origin (in different coordinate systems). Panel (a): Small open circles show a directional data set consisting of 17 points. (One data point is very near (1,0). It appears as a white dot inside the black diamond there.) This artificial data set is a singularity of an augmented directional mean data map, 
$\Phi_{1}$. Small perturbations of the data are mapped by $\Phi_{1}$ arbitrarily close to (-1,0) or (1,0), the black diamonds. However, this data set is not a singularity of another augmented directional mean, $\Phi_{2}$. Oddly, $\Phi_{2}$ locates this data set very near $(0, -1)$ (white diamond).  \linebreak
Panel (b):  Another artificial directional data set of size 17 (open circles) near (0,1). (The data points have been spread radially to improve readability.) This data set is not a singularity of $\Phi_{1}$, which locates this data set very near $(0, 1)$ (black diamond). But it is a singularity of $\Phi_{2}$: Arbitrarily small perturbations of the data set can lead $\Phi_{2}$ to locate it arbitrarily close to either (-1,0) or (1,0), the white diamonds. (See appendix \ref{Chptr:F:CircleSing.data.and.calcs} 
for a listing of the data in the figure and calculations pertaining to them.)}  \label{F:CircleSing}
 \end{figure}

Figure \ref{F:CircleSing} shows why, despite the fact that $\Phi_{1}$ has a bigger singular set one might still prefer it over $\Phi_{2}$. 
The data set shown in panel (a) is a singularity of $\Phi_{1}$. Arbitrarily small perturbations of the data can cause $\Phi_{1}$ to assign locations arbitrarily close to (-1,0) or (1,0), the black diamonds. However, this data set is not a singularity of $\Phi_{2}$, which locates this data set very near  $(0, -1)$ (white diamond). On the other hand, the data set shown in panel (a) is so diffuse that none of the diamonds seem especially objectionable as a location for the data set. In fact, the notion of ``location'' of this data set makes little sense. (Still, the location assigned by $\Phi_{2}$, viz., the augmentation point $(0, -1)$, does seem like an odd place to call the location of the data set.) So the fact that this data set is a singularity of $\Phi_{1}$ is not too troubling. (See remark \ref{R:interrogation}.)

Contrast this with the situation portrayed in panel (b). It shows a directional data set that is not a singularity of $\Phi_{1}$, which locates this data set very near $(0, 1)$ (black diamond). But this data set is a singularity of $\Phi_{2}$: Arbitrarily small perturbations of the data set can lead $\Phi_{2}$ to locate it at either of the positions indicated by the white diamonds. In this case the data are not very diffuse and it seems that their location should be somewhere near $(0,1)$, i.e., near the location assigned to it by $\Phi_{1}$. The white diamonds are at positions that seem completely wrong as locations of these data. 
 
The singular set of $\Phi_{1}$, while much larger than that of $\Phi_{2}$, is situated far from the space $\Pf$, of perfect fits for the location problem (4.95 units in Euclidean distance). Panel (b) of figure \ref{F:CircleSing}, on the other hand, shows that the singular set of $\Phi_{2}$ comes undesirably close (1.30 units in Euclidean distance) to $\Pf$. 
 
Hence, it seems that in addition to size, distance to perfect fits is an important attribute of singular sets. It turns out that size and distance to $\Pf$ are at best positively associated with each other (i.e., they increase together). 

This is illustrated in figure \ref{F:oneZeroDimP}. Each panel in this figure illustrates hypothesis testing (example \ref{Ex:hypothesis.testing}). Think of the grey regions in the two panels as consisting of data sets mapped to 1. The white portions consist of data sets mapped to 0. The boundaries of the two regions, indicated with a heavy black line, are the singular sets, denoted by $\Ss'$ here.
 
\begin{figure}
      \epsfig{file = 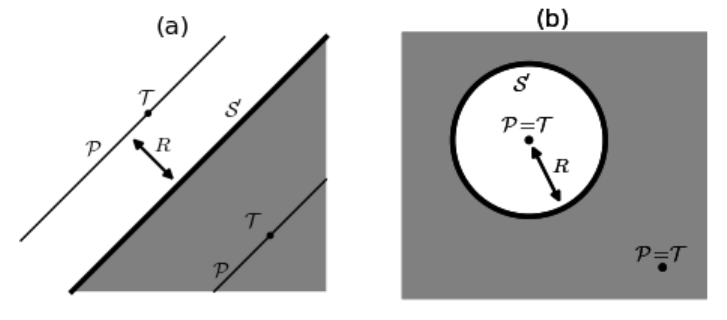, height = 2.4in, , }  
      \caption{One (a) and zero (b) dimensional $\Pf$ that is $R$ units from the singular set, 
      $\Ss'$.}    
             \label{F:oneZeroDimP}
\end{figure}

In figure \ref{F:oneZeroDimP}(a) the space $\Pf$ of ``perfect fits'' consists of two parallell lines and so is one-dimensional. The corresponding hypothesis testing method is such that the singular set is $R > 0$ units away from it at its closest. Imagine that the data space $\D$ is a large disk. As $R \downarrow 0$, the 1-dimensional volume of the singular set, denoted by $\Ss'$ here, is bounded below. Let $\Hm^{1}$ be one-dimensional Hausdorff measure. It can be thought of as measuring length, one-dimensional volume. 
 I.e., the $\Hm^{1}(\Ss')$ ``goes to 0'' like $R^{\text{codim} \, \Pf - 1} = R^{0}$.
 
In figure \ref{F:oneZeroDimP}(b) the space $\Pf$ of ``perfect fits'' consists of two points and so is zero-dimensional. The corresponding hypothesis testing method is such that, again, the singular set is $R > 0$ units away from it at its closest. 
As $R \downarrow 0$, the 1-dimensional volume of $\Ss'$ does go to 0. 
In fact $\Hm^{1}(\Ss') = 2 \pi R$. 
Thus, $\Hm^{1}(\Ss')$ again goes to 0 like $R^{\text{codim} \, \Pf - 1}$. (Panel (b) resembles example \ref{Ex:college.admissions}.)

We will see that quite generally if $R$ is the distance from $\Pf$ to $\Ss'$, then the appropriate Hausdorff measure of $\Ss'$ is no smaller than 
$\gamma R^{\text{codim} \, \Pf - 1}$, where $\gamma > 0$ depends on the geometry of the problem, but not on $(\Phi, \Ss')$. In the situation sketched in figure \ref{F:oneZeroDimP}(b), the constant is clearly $2 \pi$. In figure \ref{F:college.admissions} the constant is $\pi/2$. I have not computed $\gamma$ in nontrivial cases. 

The inequality, Hausdorff measure of $\Ss' \geq \gamma R^{\text{codim} \, \Pf - 1}$ imposes a trade-off between how small $\Ss'$ is in measure and how far it is from $\Pf$. As we have observed, fitting the data well is most important near $\Pf$. Thus, we may refer to a ``fit-instability'' tradeoff. Instability does not interfere with fit unless it is severe, the topic of the next section. 

This is another instance of the ``Sales Pitch'', remark  \ref{R:sales.pitch}, the main theme of this book: One can get global information about $\Ss'$ by looking locally, specifically within $R$ units, of $\Pf$. 
  
  \begin{remark}[Data analysis as interrogation]  \label{R:interrogation}
Here is a possible way to think of data analysis. Given a data set $x \in \D$, the statistician asks the question, ``What is the `$\F$'-ness of $x$?'', i.e., what point of $\F$ best captures the quality the statistician is looking for in the data? But for some $x$'s, the question might not make sense. 

Rarely can the statistician examine $x$ and decide for themself which $f \in \F$ best describes $x$, if any does. Instead, they usually have to delegate parts of the task to an algorithm or algorithms and the net result is a map $\Phi$ whose behavior the statistician does not completely understand. $\Phi$ can also be thought of as asking a question of $x$. A singularity of $\Phi$ can be thought of as a data set for which $\Phi$'s question does not make sense. A problem arises when $\Phi$'s question is not quite the same as the statistician's. 

For example consider the data set shown in panel ``(LS,c)'' in figure \ref{F:LS.PC.LAD.lf.plots}. It is a singularity of least squares linear regression (LS). 
LS is a regression method, so the statistician is probably asking, ``What linear function best describes the relationship between the variable plotted as abscissa and the one plotted as ordinate?''. For the panel (LS,c) data set that question does not make sense to the statistician and it does not make sense to LS either. So for this $x$ the statistician's and LS's questions are aligned. (We discussed figure \ref{F:LS.PC.LAD.lf.plots} in depth in section \ref{SSS:LF.plots}.)

Now consider panel ``(LAD,c)'' in the same figure. In this case the statistician's question is the same as before. The data set shown there is a singularity of least absolute deviation linear regression (LAD). However, this time the question ``What linear function \ldots'?'' does make sense and a plausible answer is the constant function 0. However, this data set is a singularity of LAD. The question LAD is asking does not make sense for this data set. 

\emph{So the problem with singularities is that they can occur at data sets for which the statistician's question makes sense.} (See discussion of figure \ref{F:CircleSing}.)

A premise of this book is that for data in a tight, perhaps just infinitesimal, neighborhood of $\Pf$ the statistician's question should usually make sense. (The (LS,c) data is in 
$\Pf$ but for it the statistician's question still does not make sense.) We insist that a data map $\Phi$ not have many singularities near $\Pf$, or at least not many severe ones (section \ref{SS:severity.of.sings}). But far away from $\Pf$ may be data sets for which statistician's question may not make sense. It is not so troubling if $\Phi$ has singularities there. (See discussion of figure \ref{F:CircleSing} above.)
  \end{remark}

\section{Severity}  \label{SS:severity.of.sings}
If $\Phi$ maps a neighborhood of a singularity into a small open set in $\F$ then the practical impact of singularity is small. E.g., the apparent singularity shown in panel (b) of Figure \ref{F:HeightFen} causes only a small displacement of the fitted LAD line. The displacement in panel (a) is also not large. 

To define severity in general, let $\msf{V}$ be an open cover of the codomain 
$\F$. Let $x$ be a singularity of $\Phi$. If there is no neighborhood $\clU$ of $x$ such that (s.t.)\ the closure of the image under $\Phi$ of $\clU$ lies in no set in $\msf{V}$ then $x$ is ``$\msf{V}$-severe''. If $\msf{V}$ is coarse, then $\msf{V}$-severe singularities have practical impact when data lie near them. 

The data sets plotted in the ``(c)'' column of figure \ref{F:LS.PC.LAD.lf.plots} are so called `$90^{\circ}$ singularities''. They are extremely severe. In line-fitting, every singularity of LS or PC is a $90^{\circ}$ singularity.  That is not true of LAD. In panel ``(LAD,a)'' three dashed lines descend from the apex of the triangle to the base. Except for the endpoints, every point on those lines corresponds to a singularity of LAD, but only one of them, the one displayed in ``(LAD,c)'', is $90^{\circ}$.

Severity can be connected to our results concerning Hausdorff dimension and measure of singular sets. We call that connection the ``severity trick''. 

\section{Looking ahead}  \label{S:preview}
Notation and basic definitions and results are given in chapter \ref{Chptr:basic.setup}. The relationship between the Hausdorff dimension of a set and the probability of being near it is treated in section \ref{SS:asymp.prob}. (The topic is revisitied in a more refined way in corollary \ref{C:lwr.bound.on.prob.of.being.close.to.S'}.) 

What singularity looks like when $\F$ is a metric space, but not a complete one, is discussed in remark \ref{R:.completeness.of.F}.

Instability of a map near a singularity is quantified in section \ref{SS:derivs.near.sings} and in propositions \ref{P:diam.oriented.90.degree.sing} and \ref{P:diam.of.image.of.nhbd.of.sing.of.loc.meas}. 

Possible ways of coping with singularity are considered in remark \ref{R:mitigating.sings}. 
One approach involves regularization (mentioned in section \ref{SS:calibration}) can also be used to avoid singularity, but not without problems. It is discussed in remark \ref{R:regularization.generalities}. (A specific instance of regularization is discussed in remark \ref{R:aug.mean.regularization}.)

The connection between Hausdorff measure and topology is provided by \eqref{E:cohom.and.codim}.

We first make use of algebraic topology to study singularity a general way in chapter \ref{Chptr:topology}, specifically in theorem \ref{T:Phi.star.Hr.contains.Theta.star.Hr}. It generalizes figure \ref{F:bound.cond.and.whorls}. 
It is used again in section \ref{SS:gnrl.lwr.bnd.plane.fit} and 
in the proof of theorem \ref{T:spher.loc.sing.codim.nvar+1}. 

The basic result giving a lower bound on the dimension of the singular set is given in section \ref{SS:dim.of.sing.sets}. 

Discrete $\F$, including hypothesis testing (example \ref{Ex:hypothesis.testing}; 
\ref{Ex:college.admissions} is another example) 
is analyzed in example \ref{Ex:disconnected.F} and chapter \ref{Chptr:linear.classification}. Discreteness of both $\F$ and $\D$ is discussed in remark \ref{R:discreteness}.

The ``sales pitch'' (remark \ref{R:sales.pitch}) is formalized in property \ref{Pty:agree.near.T}. Remark \ref{R:sales.pitch.in.plane.fitting} points out that plane-fitting is a domain in which the sales pitch applies well. 

The idea that one might accept a large singular set in exchange for improvement in some other aspect of data map performance (sections \ref{SS:singularity} and \ref{SS:measure.distance.to.P}) is treated in remark \ref{R:designer.sing.sets} and chapter \ref{Chptr:robst.loc.on.circle}.

Chapter \ref{Chptr:Haus.meas.of.sing.set} derives a lower bound on the Hausdorff measure of the singular set that depends on the distance from the singular set to the set $\Pf$ of perfect fits. We have already encountered this phenomenon in section \ref{SS:measure.distance.to.P}. It leads to a precise formulation of the ``sales pitch'' (remarks \ref{R:sales.pitch}, \ref{R:extended.sales.pitch}, and and property \ref{Pty:agree.near.T}). The fit-instability tradeoff discussed in section \ref{SS:measure.distance.to.P} is further discussed in remarks \ref{R:fit.instab.tradeoff} and \ref{R:designer.sing.sets}.

Severity (section \ref{SS:severity.of.sings}) is the topic of chapter \ref{Chptr:severity}. There, we develop methods that ``smooth'' away non-severe singularities leaving only severe ones. This is the basis of the ``severity trick'' (section \ref{SS:severity.of.sings}). We can often use that trick to show that sever singularities are plentiful. 
Propositions \ref{P:diam.oriented.90.degree.sing} and \ref{P:diam.of.image.of.nhbd.of.sing.of.loc.meas} give specific lower bounds on the level of instability of two classes of data maps near severe singularities. 

Chapter \ref{Chptr:sings.in.plane.fit} concerns plane-fitting, the simplest cases of which we examined in sections \ref{SSS:LF.plots} and \ref{SS:dimension}. Plane-fitting provides a striking example of the ''sales pitch'' (remarks \ref{R:sales.pitch} and \ref{R:sales.pitch.in.plane.fitting}). As suggested in section \ref{SSS:LF.plots}, we focus on the plane through the origin parallel to the graph of the function. LS is examined in some depth in section \ref{SS:lin.reg.and.LS}, LAD in section \ref{SS:LAD}, and PC in section \ref{SS:PC.plane.fitting}.

The notion of $90^{\circ}$ singularity in plane-fitting is defined in section \ref{SS:lin.combo.for.plane.fit.w/.k=nvar-1}. It does not apply to plane-fitting in general, but it does apply to linear regression (section \ref{SS:lin.reg.and.LS}), a very important statistical activity. In linear regression one computes an affine function from point clouds. The relationship between regression as a function-valued map and regression as a plane-valued map, an instance of the idea brought up in remark \ref{R:learning.and.predicting}, is discussed in section \ref{SS:functions.vs.geometry}. Severity in linear regression is the focus of section \ref{SSS:reg.coefs,near.90.degree.sings}. 

In section \ref{SS:gnrl.lwr.bnd.plane.fit}, we derive a general lower bound on the dimension of the singular sets of plane-fitting methods that should apply to \emph{any} (non-regularized; remark \ref{R:regularization.generalities}) plane-fitting method. (See remark \ref{R:general.lwr.bnd.on.plane-fitting.sing.set.dim}.)

The most basic question you can ask about a point cloud is, ``where is it?''. Statisticians use the phrase "measure of location'' to mean a data map that tells where a point cloud lies in space. Familiar examples are the arithmetic mean and the median for points clouds on the line. Those data maps are continuous, but measures of location on non-Euclidean spaces are often obligated to have singularities (Eckmann \emph{et al} \cite{bEtGpjH62.genMeans}). Chapters \ref{Chptr:spherical.location}, \ref{Chptr:aug.direct.mean}, and \ref{Chptr:robst.loc.on.circle} are devoted to measures of location on spheres, especially the circle. 

In chapter \ref{Chptr:aug.direct.mean} we examine in some depth a particular family of measures of location on spheres, the``augmented directional mean''. The data maps 
$\Phi_{1}$ and $\Phi_{2}$ discussed in section \ref{SS:measure.distance.to.P} and illustrated in figure \ref{F:CircleSing} are augmented directional means. Details concerning that plot and those data maps, including the calculation of the relative volumes of their singular sets, are presented in appendix \ref{Chptr:F:CircleSing.data.and.calcs}. 

In chapter \ref{Chptr:robst.loc.on.circle}, we examine in some depth the ``augmented directional median''. The augmented directional median can tolerate many outliers, observations far from the bulk of the data.  For this and augmented directional mean we give much attention to the volume-distance tradeoff (section \ref{SS:measure.distance.to.P}). 

Chapter \ref{Chptr:linear.classification} gives a basic account of singularity in linear classification (LC). LC describes the relationship between $k$, say, 
$\RR$-valued variables and a binary variable. From a point cloud in $\RR^{k}$, each point of which is labeled by a value of the binary variable, an LC method learns a map, $\Gamma$, that  takes (almost any) point in $\RR^{k}$ as input and returns a value of the binary variable. The hope is that, given only the $\RR^{k}$ part of a new example, the value $\Gamma$ returns is likely to be the unseen binary part of the example. 
Thus, LC is another instance of learning and prediction (remark \ref{R:learning.and.predicting}). In LC both the training and predicting steps have singularities. 

Here, as in sections \ref{SS:D.T.plane.fit} and \ref{S:exactness.of.fit} the test pattern space, $\T$, in LC is much smaller than the space $\Pf$ of perfect fits. 

In section \ref{SS.lin.discr.anal}, we study linear discriminant analysis, a common LC method. Specific examples are given there of the role of subjective judgment in data analysis, an issue brought up in section \ref{S:topology}.

There are eight appendices. Appendix \ref{Chptr:misc.proofs} resolves some technical issues that arise in the main body of the book. Appendix \ref{Chptr:F:CircleSing.data.and.calcs} lists the data used in figure \ref{F:CircleSing} and explains how the value 6.6 billion mentioned in section \ref{SS:measure.distance.to.P} was calculated. 
Appendices \ref{Chptr:Lip.Haus.meas.dim} reviews the topics of Lipschitz maps and Hausdorff dimension and measure. Appendix \ref{Chptr:basics.of.simp.comps} reviews simplicial complexes. In appendix \ref{Chptr:polyhed.approx} we develop some technicalities used in chapter \ref{Chptr:Haus.meas.of.sing.set} concerning approximating a function continuous off a compact set by one continuous off a polyhedron. Appendix \ref{Chptr:LAD.technicalities} fills in some technical gaps for the specific topic of ``least absolute deviation'' linear regression (subsection \ref{SS:LAD}). Appendix \ref{Chptr:rob.loc.circle.cones.appendix} concerns the set of perfect fits relevant to resistant measures of location on the circle, the subject of chapter \ref{Chptr:robst.loc.on.circle}. 

Just before the bibliography are acknowledgements of the help I have received through the years on this project.

\chapter{Preliminaries}   \label{Chptr:basic.setup}	
Let $\Phi$ be a data map, i.e., $\Phi$ takes data sets, $x$, in a ``data, or sample space'', $\D$, and maps them to decisions, classifications, estimates, descriptions, or features in a feature or parameter space, $\F$. (See section \ref{S:intro.to.intro}; nondeterministic data maps are discussed in section \ref{SS:nondeterministic.data.maps}.)  In this book we make the following blanket assumption.
	\begin{equation}  \label{E:D.metric.F.normal}
	  \D \text{ is a separable, pathwise connected metric space and $\F$ is normal.} 
	\end{equation}
For example, if $\D$ is a finite dimensional topological manifold (Boothby \cite[Definition (3.1), p.\ 6]{wmB75}) then it is metrizable and separable (Boothby \cite[Theorem (3.6), p.\ 9]{wmB75} and Simmons \cite[p.\ 100]{gfS63}). 

In the examples considered in this book $\F$ is a finite dimensional manifold. But it would be of interest to consider infinite dimensional $\F$. 

  \begin{remark}[Sample]  \label{R:sample}
The points of $\D$ are ``data sets'' or ``realizations''. A data set is usually a (finite) list of points, $x_{1}, x_{2}, \ldots, x_{n}$ in some space $\X$. Usually, the order of the points does not matter and the same point of $\X$ may appear more than once. Such may matter (chapters \ref{Chptr:spherical.location}, \ref{Chptr:aug.direct.mean}, and \ref{Chptr:robst.loc.on.circle}). We call such a list a ``sample''. So a sample is a multiset. 
  \end{remark}

  \begin{remark}[Basic notation] \label{R:basic.notation}
In this book we frequently use the following notation. Let $X$ be a metric space
and $r > 0$. Then we define
	\begin{equation}  \label{E:ball.defn}
		B_{r}(x) := \bigl\{ y \in X : \phi(y,x) < r \bigr\},   
		  \quad  \overline{B_{r}(x)} := \bigl\{ y \in X : \phi(y,x) \leq r \bigr\},   
	\end{equation}
where $\phi$ is the metric on $X$. Variations on this notation will also appear. For example,
    \begin{equation}     \label{E:Euc.ball.defn}
		B_{r}^{m}(x) := \bigl\{ y \in \RR^{m} : |y - x| < r \bigr\}, \quad
		 \overline{B_{r}^{m}(x)} 
		    := \bigl\{ y \in \RR^{m} : |y - x| \leq r \bigr\}
    \end{equation}
	
Another common notation is the following. Let $f : S \to Y$ and let $T \subset S$. We write
	\begin{equation}  \label{E:restriction.notation}
		f \restriction_{T} \text{ is the restriction of } f \text{ to } T.
	\end{equation}
If $f$ is not defined everywhere on $S$ we still might write $f : S \partlyto Y$.
  \end{remark}
  
Suppose $\D'$ is a dense subset of $\D$ and $\Phi : \D' \to \F$. 
Thus, $\Phi : \D \partlyto \F$. If $\clU \subset \D$ we sometimes write 
$\Phi(\clU) := \Phi ( \clU \cap \D' )$ and call that the ``image of $\clU$ under $\Phi$."

Note that if $\Ss \subset \D$ then $\Ss$ is also a separable metric space. 
\emph{(Proof:} $\Ss$ is obviously a metric space. Since $\D$ is separable, it is second countable, by Simmons \cite[Theorem C, p.\ 100]{gfS63}. Let $n = 1, 2, \ldots$ be arbitrary. Then, by Lindel\"of's theorem, Simmons \cite[Theorem A, p.\ 100]{gfS63}, we have that 
$\bigl\{ B_{1/n}(x) ; x \in \Ss \bigr\}$ has a countable subcover 
$\bigl\{ B_{1/n}(x_{ni}) ; i = 1, 2, 3, \ldots \bigr\}$. 
The set $\{ x_{ni} \in \Ss : i, n = 1, 2, \ldots \}$ is a countable dense subset of $\Ss$.)

Say that a manifold or map is ``smooth'' if it is $C^{\infty}$. $\D$ will often be a Riemannian manifold (Boothby \cite[Definition (2.6), p.\ 184]{wmB75}), in which case, for convenience, we include among its properties that it is smooth. 

A ``singularity'' of $\Phi$ (w.r.t.\ $\D'$) is a data set, $x \in \D$, at which the limit, $\lim_{x' \to x} \Phi(x')$, does not exist. (The limit is taken through the dense subset $\D'$.)  So a singularity, $x$, is like a discontinuity except that $\Phi$ does not have to be defined at $x$ and even if it is defined at $x$, the value of $f(x)$ is irrelevant. The set of all singularities is the ``singular set,'' $\Ss$, of the data map, $\Phi$. 

 \begin{remark}[Completeness and interpretation of singularity]  \label{R:.completeness.of.F}
Consider the assumption 
	\begin{equation*}  \label{E:F.is.metric}
		\text{$\F$ is a metric space.} 
	\end{equation*}
Then a natural interpretation ``$x \in \Ss$ is a singularity'' is that small changes in $x' \in \D'$ near $x$ can have relatively large effects on $\Phi(x')$. But what if $\F$ is not complete?  Suppose $x_{n} \to x \in \Ss$ and $\{ \Phi(x_{n}) \}$ is Cauchy but does not converge in $\F$. It might be possible to augment $\F$ by an object $f$ in a sensible fashion so that then 
$\Phi(x_{n}) \to f$. Assume this has been done whenever possible. However, there might be some cases in which there is no sensible object to which $\{ \Phi(x_{n}) \}$ is ``trying'' to converge. E.g., $\F$ might be a function space and $\{ \Phi(x_{n}) \}$ is ``trying'' to converge to a multi-valued ``function''. (An example is given in section \ref{SS:functions.vs.geometry}.) In that case, one might prefer to not to define a limiting value of $\{ \Phi(x_{n}) \}$. In summary, if $x \in \D$ is a singularity then either small changes in $x' \in \D'$ arbitrarily near $x$ can have relatively large effects on $\Phi(x')$ or near $x$ one can find $x' \in \D'$ such that (s.t.)\ $\Phi(x')$ is arbitrarily ``strange''. Our focus in this book is on cases in which $\F$ is complete, in fact, compact.
 \end{remark}   
 
Statisticians often analyze the variability of a data map \emph{via} a Taylor expansion. Clearly, a Taylor expansion can never capture singularity. 
(But perhaps approximation by rational functions might?)

Let $\Ss'$ be a closed superset of $\Ss$. (In chapter \ref{Chptr:severity}  we will see that the closed set $\Ss'$ can often be replaced by a closed \emph{subset} of $\Ss$.) Assume $\Phi$ is defined and continuous on $\D \setminus \Ss'$. In chapter \ref{Chptr:topology} we will see that if $\F$ has nontrivial homology and $\T \subset \D$ is rich enough that the restriction of $\Phi$ to $\T \setminus \Ss'$ probes that homology then 
$\Phi$ can thereby be forced to have singularities somewhere in $\D$, not necessarily 
in $\T$. 

The singular set depends on the dense set $\D'$:

\begin{example}   \label{Ex:bad.x'}
Let $\D = (0,1)$ and $\F = \RR$. Let $q_{1}, q_{2}, \ldots $ be the rational numbers in $\D$. Define $\Phi : \D \to \F$ as follows. If $x \in \D$ is irrational then $\Phi(x) := 0$. For $i = 1, 2, \ldots$, write $q_{i} = a/b$ in lowest terms. If $b$ is odd, then $\Phi(q_{i}) := 1$. If $b$ is even, then $\Phi(q_{i}) := 2$. 
First, let $\D' = \{ q_{1}, q_{2}, \ldots \}$. Then the singular set is $\Ss = \D$.\footnote{Any $x \in \D$ can be approximated by a fraction $a/b$ in lowest terms with $b$ even. Let $\ell = 1, 2, \ldots$. Then we can make the fraction $q := 2^{\ell} a/(2^{\ell} b + 1)$,  arbitrarily close to $a/b$ by taking $\ell$ sufficiently large. Now, $2^{\ell} a$ and $2^{\ell} b + 1$ might not be relatively prime, but it is easy to see that in lowest terms $q$ has odd denominator. Thus, arbitrarily close to $x$ one can find points at which $\Phi$ is $1$ and points at which $\Phi$ is 2.} 
On the other hand, if $\D'$ is the set of irrational numbers in $\D$, then $\Ss = \varnothing$.
\end{example}

In practice, there is usually a natural choice of $\D'$ satisfying the following.
   \begin{equation}   \label{E:D'.dense.Phi.cont.on.D'}
      \D' \text{ is dense in } \D \text{ and } \Phi \text{ is continuous on } \D' \; 
      ( \text{so }  \; \D' \cap \Ss = \varnothing).
   \end{equation}   
Under \eqref{E:D'.dense.Phi.cont.on.D'} all isolated points of $\D$ belong to $\D'$ and if $x' \in \D'$ 
is not isolated, then any neighborhood of $x'$ contains another point of $\D'$. We have the following. (``$\setminus$'' indicates set theoretic subtraction.) 
   \begin{lemma}  \label{L:extend.Phi.to.D.less.S}
      Suppose $\Phi : \D' \to \F$ is a data map with singular set $\Ss$ w.r.t.\ $\D'$.
      Suppose \eqref{E:D'.dense.Phi.cont.on.D'} holds. Then $\Ss$ has empty interior. 
      Define $\hat{\Phi}: \D \setminus \Ss \to \F$ as follows. Let
         \begin{multline*}
            \hat{\Phi}(x) = 
               \begin{cases}
                  \Phi(x), &\text{ if $x \in \D \setminus \Ss$ is isolated,} \\
                    \lim_{x' \to x; \, x' \in \D'} \Phi(x) , 
                      &\text{ if } x \in \D \setminus \Ss \text{ is not isolated.}
               \end{cases}  \\
         \end{multline*}
Then $\hat{\Phi}$ is defined and continuous on $\D \setminus \Ss$. It is the unique continuous extension of $\Phi$ to $\D \setminus \Ss$ and the singular set of $\hat{\Phi}$ w.r.t.\ $\D \setminus \Ss$ is $\Ss$.

Suppose $\tilde{\D}$ is another dense subset of $\D$ on which $\hat{\Phi}$ is defined, i.e., $\tilde{\D} \subset \D \setminus \Ss$, and let $\tilde{\Ss}$ be the singular set 
of $\hat{\Phi}$ w.r.t.\ $\tilde{\D}$. Then $\tilde{\Ss} \subset \Ss$.
   \end{lemma}
(For proof see appendix \ref{Chptr:misc.proofs}.)

Thus, we may typically assume 
   \begin{equation}  \label{E:D'.=.D.less.S}
      \D' = \D \setminus \Ss, \; \Phi \text{ is continuous on } \D', 
        \text{ and $\D'$ is dense in $\D$}.
   \end{equation} 
(Sometimes it may be convenient to allow $\Phi$ to be defined on a set larger than 
$\D'$.)

\begin{remark}[``Damned if you do and damned if you don't'']  \label{R:Scylla.and.Charybdis}
Let $\D' \subset \D$ with $\Phi$ defined on $\D'$. One hopes (1) $\D'$ is large and (2) $\Phi$ behaves well on $\D'$. 
E.g., $\D'$ might be the collection of data sets
at which $\Phi$ is defined uniquely. It may not be easy to prove that \eqref{E:D'.dense.Phi.cont.on.D'} holds for this $\D'$. One can then proceed as
follows.  Assume that \eqref{E:D'.dense.Phi.cont.on.D'} holds and on that basis use the results in this book 
to show that some generally bad things happen. Alternatively, \eqref{E:D'.dense.Phi.cont.on.D'} fails, which in itself can be bad. 
\end{remark}

An example of a $\D'$ to which the preceding remark might apply is provided by data maps defined by optimization. These are common in data analysis (e.g., remark \ref{R:regularization.generalities}). Recall our blanket assumption \eqref{E:D.metric.F.normal}.

	\begin{lemma} \label{L:data.maps.defined.by.opt}
Suppose $g : \F \times \D \to \RR$ and, given $x \in \D$, $\Phi(x)$ is defined to be the point 
$f _{0}\in \F$ s.t.\ $f = f _{0}$ minimizes $f \mapsto g(f,x)$, whenever $f _{0}$ exists uniquely.
	\begin{enumerate}
		\item Suppose $\F$ is compact and $g$ is continuous on $\F \times \D$. Let 
$\D'_{\ref{I:cmpct.opt}} \subset \D$ be the collection of data sets $x' \in \D$ s.t.\ $\Phi(x')$ is defined. Suppose $\D'_{\ref{I:cmpct.opt}} \neq \varnothing$. 
Then $\Phi$ is continuous on $\D'_{\ref{I:cmpct.opt}}$. \label{I:cmpct.opt}
		\item More generally, even if $\F$ is not compact or $g$ is not continuous, the following holds. Let $\D'_{\ref{I:unif.opt}} \subset \D$ be the collection of data sets 
$x' \in \D$ with the following property. There exists $f_{0} \in \F$ (depending on $x'$) s.t.\  for any neighborhood, $G$, of $f_{0}$ there is a neighborhood $\clU \subset \D$ of $x'$ s.t.\
			\begin{equation}  \label{E:outside.G.g.is.big.unif}
			     \text{For every } x \in \clU, \;  \inf_{f \notin G} g(f, x) > g(f_{0}, x).
			\end{equation} 
So $\Phi(x')$ exists (in particular is unique) and equals $f_{0}$. Suppose $\D'_{\ref{I:unif.opt}} \neq \varnothing$. Then $\Phi$ is continuous on $\D'_{\ref{I:unif.opt}}$.  
                 \label{I:unif.opt}
	\end{enumerate}
	\end{lemma}
(For proof see appendix \ref{Chptr:misc.proofs}.)

  \begin{remark}[Distance based methods] \label{R:dist.based.mthds}
In \cite{spEsM92.distanceScope} it is proved that, in \emph{principle}, a large collection of data maps can be thought of as defined as minimizers of a distance defined by a metric. Specifically, suppose $\Phi : \D \partlyto \F$ and  $\Phi$ imbeds 
$\Pf \subset \D$ into $\F$ in such a way that $\Phi(\D') = \Phi(\Pf)$, where $\D'$ is the subset of $\D$ on which $\Phi$ is defined and continuous. As usual, we assume 
$\D'$ is dense in $\D$. Then we can identify $\Phi(D')$ with $\Pf$ and regard $\Phi$ as mapping, a retraction (remark \ref{R:retractions}), $\D'$ into $\Pf$. Recall that by \eqref{E:D.metric.F.normal}, $\D$ is metrizable and second countable. Assume it is also locally compact. Let $\Ss := \D \setminus \D'$ (see lemma \ref{L:extend.Phi.to.D.less.S}) and suppose $\overline{\Ss}$ \emph{and 
$\overline{\Pf}$ are disjoint}. Then there exists a metric, $\rho$, on $\D$ that generates the topology on $\D$ and if $x \in \D'$ then $\Phi(x)$ is the unique $\rho$-closest point of $\Pf$ to $x$.

In chapter \ref{Chptr:topology} and beyond we will almost always assume that, if $\overline{\Ss}$ and $\overline{\Pf}$ are not disjoint, their intersection is at least small. 
  \end{remark}

  \begin{remark}[Mitigating Singularity] \label{R:mitigating.sings}
We indicate three approaches.

First, one can try to choose a data analytic method, $\Phi$, that is appropriate for the problem at hand but whose singular set lies in a region of $\mcl{D}$ of low probability. This choice can be made before seeing the data. But there may not be any such 
$\Phi$. See remark \ref{R:dnsity.contrs.as.data.spaces}. 

It is tempting to try to avoid singularities by choosing $\Phi$ on the basis of the data, but this choice must be made cautiously. To see this, suppose one starts with a family, $\blds{\Phi} = \{ \Phi_{\alpha} : \D' \to \F, \, \alpha \in A \}$, of procedures appropriate to the problem at hand. E.g., each $\Phi_{\alpha}$ is calibrated w.r.t.\ some simple, ``rich'' space of perfect fits (subsection \ref{SS:calibration}). Given data, $x$, choose some member, $\Phi_{\alpha(x)}$, whose singular set is remote from $x$. Summarize the data by $\Phi^{\ast}(x) := \Phi_{\alpha(x)}(x)$. At first glance it would appear that $\Phi^{\ast}$ has good singularity properties. However, by piecing together the procedures 
in $\blds{\Phi}$ one may inadvertently create singularities. In particular, the theory of chapter \ref{Chptr:topology} may apply to show that $\Phi^{\ast}$ has singularities. For example,
$\alpha(x)$, the index that selects the member of $\blds{\Phi}$ based on the data $x$, may be a piece-wise constant function of $x \in \mcl{D}$. $\Phi^{\ast}$ may have singularities at some of the jumps in $\alpha$. If all the methods in $\blds{\Phi}$ are calibrated w.r.t.\ the same set of perfect fits and ``standard'' $\Sigma$ then, as demonstrated in subsequent chapters, topology may force $\blds{\Phi}$ to have a large singular set no matter how it is constructed. 

A second approach to countering the singularity problem is through the use of ``diagnostics''. For example, 
suppose $\Phi$ has a closed singular set $\Ss$. (In chapter \ref{Chptr:severity} we will see that this is not an unreasonable assumption.) Let $\mcl{C}$ denote the ``cone over $\F$'' defined as follows. (See subsection \ref{SSS:conical.fibers}.)  Start with the Cartesian product, $\F \times [0, 1]$. The cone, $\mcl{C}$, is then the space we obtain by identifying the set $\F \times \{ 0 \}$ to a point, the ``vertex'', $v$. 

If $X$ is a metric space with metric $d$, $x \in X$, and $A,B \subset X$, let 
    \begin{multline}  \label{E:set.distances}
       dist(x, A) := dist_{d}(x, A) := \inf \bigl\{ d(x, y) :  y \in A \bigr\} \text{ and } \\
         dist(A,B) := dist_{d}(A,B) := \inf \bigl\{ dist(y, B) : y \in A \bigr\} \\
           =  \inf_{y,z} \bigl\{ d(y,z) :  y \in A, z \in B \bigr\} 
             =  \inf \bigl\{ dist(z, A) : z \in B \bigr\} .
    \end{multline}
(See Munroe \cite[p.\ 12]{meM71.meas.thy}.) Also recall that the ``diameter'' of $A$ is defined to be
    \begin{equation}  \label{E:diam.of.set}
          diam(A) = \sup \bigl\{ d(x,y) \geq 0: x, y \in A \bigr\} .
    \end{equation}
 (See Munroe \cite[p.\ 12]{meM71.meas.thy} .)

If $x \in \mcl{D}$, let $\rho(x) = \text{dist}(x, \Ss)$, the distance from $x$ to $\Ss$. 
($\D$ is a metric space by \eqref{E:D.metric.F.normal}.) Now define a new data analytic procedure, $\Phi^{\#}$ taking values in $\mcl{C}$, as follows. 
   \[
      \Phi^{\#}(x) =
         \begin{cases}
            \left( \Phi(x), \frac{\rho(x)}{1 + \rho(x)} \right) \in \mcl{C}, \text{ if } x \notin \Ss, \\
            v, \text{ if } x \in \Ss.
         \end{cases}
   \]  
The quantity $\delta(x) := \rho(x)/[1 + \rho(x)]$ is a ``diagnostic'' for $\Phi$ at $x$. The augmented data summary $\Phi^{\#}$ is actually a continuous function on $\mcl{D}$ that contains all the information in $\Phi$ (assuming $\Phi(x)$ is
meaningless for $x \in \Ss$). However, computing $\delta$ may be quite difficult. But  diagnostics related to the distance to the singular set might alert the data analyst that a singularity is near. See Belsley \cite{daB91}. The ``$p$-value'', example \ref{Ex:hypothesis.testing}, is an example of such a diagnostic. But in the hypothesis testing setting, a binary decision has to be made so in that context the ``$p$-value'' is not very helpful. (The ``Bayes factor'', Berger \cite[Definition 6, p.\ 146]{joB85.BergerBayes} is an alternative.) Finding a diagnostic related to distance seems impossible if the data map has a subjective component (section \ref{S:topology}). But even when $\delta$ cannot be computed, diagnostics that merely warn that a singularity is nearby are helpful. An obvious diagnostic is a measure of how much 
$\Phi(x)$ fluctuates as $x$ is perturbed, perhaps by a number of random perturbations. 

A third, and widely used, method for mitigating singularities is through regularization. But regularization also has it perils (remark \ref{R:regularization.generalities}). 
Randomization, i.e., using nondeterministic data maps (section \ref{SS:nondeterministic.data.maps}), seems likely to aggravate, not mitigate, singularity.
  \end{remark}

  \begin{remark}[Distribution shift]
Suppose a statistical model is fitted, using a data map $\Phi$, from a data set $x$ drawn from a probability distribution, $P$. Suppose one wishes to apply that model to data drawn from a slightly different distribution, $P'$. Might that model be wildly wrong in the new context? If the distance between $P$ and $P'$ is small, w.r.t.\ the Prohorov metric (Billingsley \cite[pp.\ 237--239]{pB68.ConvProbMeas})) say, then samples from $P'$ should be ``close'' to those $x$. Thus the question of the stability of the fitted model to distribution shift amounts to considering the stability of the data map $\Phi$ at $x$. 
  \end{remark}

  \begin{remark}[Duality]  \label{R:Duality}
Let $x \in \D'$. It might be that two data maps, $\Phi_{1}$ and $\Phi_{2}$, on $\D'$ are very similar (e.g., are solutions to very similar optimization problems), yet $\Phi_{1}(x)$ and $\Phi_{2}(x)$ are quite different. For $x$ may be close to the singular set, $\Ss_{1}$,  
of $\Phi_{1}$ and the singular set of $\Phi_{2}$ may be different from $\Ss_{1}$. In this case one might say that 
\emph{$x$ has a singularity near $\Phi_{1}$}. Such singularities are important in applied Statistics. 
It is hard to trust a statistical analysis in which a small change in  the model leads to a large change in the fitted point of $\F$.

More generally, consider the space $\mbf{M} = \D \times \blds{\Phi}$, 
where $\blds{\Phi}$ is a collection of data maps with the same range space, $\F$. Consider the singularities of the pairing that takes 
$(x, \Phi) \in \mbf{M}$ to $\Phi(x)$. Might this pairing be important in model selection? A difficulty here is choosing a topology for $\blds{\Phi}$. The most practical, but least interesting, choice might be the discrete topology.

This is connected to issue of reproducibility of data analysis (D'Amour \emph{et al} \cite{aD'AkHetAL2020}) and to the notion of stability discussed in Yu and Kumbier \cite{bYkK2019.VeridicalDataScience}. In fact, we can consider the joint stability 
of $( \Phi, x ) \mapsto \Phi(x)$ as a measure of the stability of the whole ``data science life cycle'' discussed in Yu and Kumbier. 
  \end{remark}

\emph{Prima facie}, the bigger $\Ss$ is, the poorer is the conditioning of $\Phi$. One way to measure the size of the singular set is by its dimension. We use Hausdorff dimension. (See appendix \ref{Chptr:Lip.Haus.meas.dim}.) A stronger notion of dimension is covering or \v{C}ech-Lebesgue dimension (Engelking \cite[Definition 1.6.7, p.\ 54]{rE78.DimThy}), which in our case (\eqref{E:D.metric.F.normal}) is the same as inductive dimension (Engelking \cite[Theorem 1.7.7, p.\ 65]{rE78.DimThy}, Hurewicz and Wallman \cite[pp. 6--7, 67]{wHhW48.DimThy}; we showed above that $\Ss$ is separable metric). Denote inductive dimension by $ind$.
Let $X$ be a separable metric space. Then, by Hurewicz and Wallman \cite[Theorem VII 2, p.\ 104]{wHhW48.DimThy},
   \begin{equation}  \label{E:n.dim.Haus.meas.>0.if.topo.dim.geq.n}
      ind \, X \geq n \text{ implies } \Hm^{n}(X) > 0 \text{ so } \dim X \geq n, 
        \quad n = 0, 1, 2, \ldots.
   \end{equation}
where ``$\Hm^{n}$'' denotes $n$-dimensional Hausdorff measure and ``$\dim$'' denotes Hausdorff dimension (appendix \ref{Chptr:Lip.Haus.meas.dim}). 

Let $X$ be a compact metric space and let $k$ be a nonnegative integer. Let $G$ be an abelian group and let $\check{H}^{k}(X; G)$ be the $k$-dimensional \v{C}ech cohomology group of $X$ with coefficients in $G$ (Munkres \cite[\S 73 ]{jrM84}; Dold \cite[Chapter VIII, chapter 6]{aD95.alg.topol}). We have:
   \begin{subequations}  \label{E:cohom.and.codim}
     \begin{gather}  
        \text{If } \check{H}^{k}(X; G) \ne 0   \label{E:k.cohom.not.0} \text{ then } \\
           \Hm^{k}(X) > 0. \; \text{ In particular, } \dim X \geq k.  
             \label{E:big.dim.small.codim} 
     \end{gather}
   \end{subequations}
This follows from Hurewicz and Wallman \cite[statement F, p.\ 137]{wHhW48.DimThy} and
\eqref{E:n.dim.Haus.meas.>0.if.topo.dim.geq.n}. (See also Dold \cite[Chapter VIII, Remark 10.4, pp.
 309--310]{aD95.alg.topol}.) Note that since $\check{H}^{k}(X; G)$ is a topological invariant, we have that $ \Hm^{k}(X) > 0$ if $\check{H}^{k}(X; G) \ne 0$ no matter for what metric on $X$ it is computed, providing it is compatible with the topology.
 
 In the interest of explicitness, we posit
     \begin{equation}  \label{E:(co)hom.of.empty.set.is.0}
      \text{The homology and cohomology of the empty set is trivial in all dimensions.}
    \end{equation}
Together, \eqref{E:Haus.measure.of.empty.set.is.0} and the preceding do not conflict with \eqref{E:cohom.and.codim}. 
 
  \begin{example}[Singular set and closure can have different dimensions]  \label{E:different.dimensions}
First, we construct a probability measure on $[0,1]$. Let $k = 1, 2, \ldots$. 
At the points $j/2^{k}$ \emph{for odd } 
$j = 1, 3, 5, \ldots, 2^{k} - 1$, put mass $2^{1 - 2k}$. That totals $2^{k-1} \times 2^{1 - 2k} = 2^{-k}$. Summing over $k$ yields 1 so the sum of all these masses is a probability measure. Note that the collection $\{ j/2^{k} : k = 1, 2, \ldots \text{ and } j = 1, 3, 5, \ldots, 2^{k} - 1 \}$ is precisely the set of \emph{dyadic rationals} in $(0,1)$, i.e., numbers of the form $j /2^{k}$ for $k = 1, 2, \ldots$ 
and $j = 1, \ldots 2^{k} - 1$, including even $j$ in that range. The dyadic rationals are dense in $[0,1]$.

Let $\Phi : [0,1] \to \F := [0,1]$ be the distribution function of this probability measure (Chung \cite[Chapter 1]{klC74.prob}). I.e., $\Phi(t)$ is the probability assigned to $[0,t]$ ($t \in [0,1]$). In particular, $\Phi$ is increasing. Let $\D' \subset [0,1]$ consist of all points in $[0,1]$ that are \emph{not} dyadic rationals. $\D'$ is also dense in $[0,1]$. \emph{Claim:} $\Phi$ is continuous on $D'$. 
Let $x \in (0,1) \cap \D'$, so $x$ is not a dyadic rational. For sufficiently large $k = 1, 2, \ldots$ we have $2^{-k} < x < 1 - 2^{-k}$. Hence, there exists $j = 2, \ldots 2^{k}-1$ 
s.t.\ $(j-1) 2^{-k} < x < j 2^{-k}$. Let $t \in (x, j 2^{-k})$ be arbitrary. In the range 
$\bigl( (j-1) 2^{-k}, t \bigr]$ there are no points $\ell 2^{-m}$ with $m \leq k$. 
Thus, $\Phi(t) - \Phi \bigl( (j-1) 2^{-k} \bigr)$ is the sum of fractions $\ell 2^{1-2m}$ for a range of odd $\ell$ and $m > k$. Therefore, a crude upper bound is
    \begin{equation}
       \Phi(t) - \Phi \bigl( (j-1) 2^{-k} \bigr) \leq \sum_{m > k}^{\infty} 2^{m-1} \times 2^{1 - 2m}
         = \sum_{m > k}^{\infty} 2^{-m} = 2^{-k}.
    \end{equation}
Since $\Phi$ is increasing, if 
$y \in \bigl( (j-1) 2^{-k}, t \bigr)$ then $\Phi(y) \in \bigl[ \Phi \bigl( (j-1) 2^{-k} \bigr), \Phi( t) \bigr]$. In particular, 
$\Phi(x) \in \bigl[ \Phi \bigl( (j-1) 2^{-k} \bigr), \Phi( t) \bigr]$. It follows that 
$\bigl| \Phi(y) - \Phi(x) \bigr| < 2^{-k}$. Hence, $\Phi$ is continuous at $x$. A simpler argument shows $\Phi$ is continuous at 0 and 1. The claim is proved. 

Let $\Ss$ be the set of all dyadic rationals in $(0,1)$. If $x \in \Ss$ then we can write $x = j 2^{-k}$ with $k = 1, 2, \ldots$ and odd $j = 1, 3, 5, \ldots, 2^{k} - 1$. We have $\Phi(x) - \Phi(x-) = $ mass at $x = 2^{1 - 2k}$, where $\Phi(x-)$ is the limit of $\Phi$ at $x$ from the left. (It exists because $\Phi$ is increasing.) Therefore, $\Phi(x+) - \Phi(x-) > 0$ ($\Phi(x+)$ is the right limit). Hence, $\Ss$ is the singular set of $\Phi$ w.r.t.\ $\D'$ and $(\Phi, \D', \Ss)$ satisfies \eqref{E:D'.=.D.less.S}. By \eqref{E:dim.of.whole.=.max.dim.of.parts}, $\dim \Ss = 0$. Of course, the closure, $\overline{\Ss}$, is just the unit interval and so has dimension 1, 
by \ref{E:Haus.dim.s-manif.=.s}.
  \end{example}
 
  \begin{remark}[Reliance on the Axiom of Choice]  \label{R:axiom.of.choice}
At various points in this book choices will need to be made. Permission to make them is granted by the Axiom of Choice. Some of us are squeamish about relying on the Axiom of Choice. The results in this book are intended to be applied to specific real data maps that might actually be used in practice. We can be comforted by the assumption that such data maps would be sufficiently nonpathological that the aforementioned choices could be made algorithmically. 
  \end{remark}

\section{Asymptotic probability of being near $\Ss$}   \label{SS:asymp.prob}
Let $\D$ be a $d$-dimensional Lipschitz manifold, where $d$ is a positive integer. I.e., each point of $\D$ has a coordinate neighborhood s.t., w.r.t.\ the metric on $\D$, the local coordinate functions of $\D$ and their inverses are Lipschitz functions (appendix \ref{Chptr:Lip.Haus.meas.dim}). (By lemma \ref{L:coord.maps.are.Lip}, part \eqref{I:phi.psi.Lip}, a Riemannian manifold, for example, is a Lipschitz manifold.) Let $\Phi : \D \partlyto \F$ be a data map with the singular set $\Ss \subset \D$. By definition of singularity, a small perturbation of a data set near $\Ss$ of a data map can wildly change the output of $\Phi$. (This is further explored in section \ref{SS:derivs.near.sings}). In this section we learn something about the probability of getting data near $\Ss$. This section is reminiscent of Blum \emph{et al} \cite[Chapter 13]{lBfCmSsS98.realcompute}. 

Let $\Rcl$ be a compact subset of $\D$. We are interested in the case 
of $\Rcl \subset \Ss$, but $\Rcl$ may be any compact subset of $\D$. For $\delta > 0$ recall that the ``$\delta$-parallel body'' of $\Rcl$ is
	\begin{equation}  \label{E:parallel.body}
	           \Rcl^{\delta} = \{ x \in \D : \text{dist}(x, \Rcl) \leq \delta \}.
	\end{equation}
(Falconer \cite[p.\ 41]{kF90}); \ref{E:set.distances} above.) I expect the following is well known. In fact, it is reminiscent of Weyl's  tube formula (Gray \cite{aG04.Tubes}).
 
  \begin{prop}  \label{P:tube.about.sing.set}
Let $\D$ be a $d$-dimensional Lipschitz manifold and let $\Rcl \subset \D$ be compact. 
Suppose the positive multiplicative constant, $\omega_{s} \in (0, \infty)$, 
in the definition of $\Hm^{s}$ is continuous in $s > 0$ (e.g., as in \eqref{E:vol.of.unit.ball}). Let $r := \dim \Rcl$ so $r \leq d$. Assume $r > 0$. Recall that 
$\text{codim} \, \Rcl := d - \dim \Rcl = d - r$. Let $\epsilon \in (0, r)$. We have the following.
     \begin{equation}  \label{E:Hm.Rdelta.geq.delta.to.codim}
        \Hm^{d}(\Rcl^{\delta}) \geq \delta^{(\text{codim} \, \Rcl) + \epsilon}, 
           \; \text{for } \delta \in (0,1) \text{ sufficiently small,}
     \end{equation}
 If $\Hm^{r}(\Rcl) < \infty$ then there is a constant, $C > 0$, depending only on $\Rcl$ and $\D$ s.t.\ 
     \begin{equation}  \label{E:Hm.Rdelta.geq.delta.to.codim.times.Hm.R}
        \Hm^{d}(\Rcl^{\delta}) 
           \geq C \, \delta^{\text{codim} \, \Rcl} \, 
              \Hm^{\dim \Rcl}(\Rcl), 
                \; \text{for } \delta \in (0,1) \text{ sufficiently small}.
     \end{equation}
 \end{prop}
For proof see appendix \ref{Chptr:misc.proofs}. Corollary \ref{C:lwr.bound.on.prob.of.being.close.to.S'} has more to say about \eqref{E:Hm.Rdelta.geq.delta.to.codim.times.Hm.R}. 

From the proof of proposition \ref{P:tube.about.sing.set} it appears that the constant $C$ in \eqref{E:Hm.Rdelta.geq.delta.to.codim.times.Hm.R} might often be computable. In theorem \ref{T:lwr.bnd.on.Haus.meas} we get a relative lower bound on $\Hm^{\dim \Rcl}(\Rcl)$ under some circumstances. Corollary \ref{C:lwr.bound.on.prob.of.being.close.to.S'} refines this proposition.

Let $P$ be a probability distribution on $\D$, absolutely continuous w.r.t.\ $\Hm^{d}$. Suppose the density $dP/d\Hm^{d}$ is continuous and nowhere vanishing. Then the density is bounded below on any relatively compact neighborhood of $\Rcl$. Let $x$ be a random element of $\D$ with distribution $P$. It follows from \eqref{E:Hm.Rdelta.geq.delta.to.codim} that the probability that $x$ is within $\delta$ of $\Rcl$ goes to zero more slowly than $\delta^{(\text{codim} \, \Rcl) + \epsilon}$ for any $\epsilon > 0$. Corollary \ref{C:lwr.bound.on.prob.of.being.close.to.S'} again refines this. In figure \ref{F:sing.dists.cdfs} above the probability distribution of the distance to the singular sets of three plane fitting methods (chapter \ref{Chptr:sings.in.plane.fit}) is exhibited in a toy case. In each case the probability that $x$ is within $\delta$ of $\Ss$ goes to 0 like 
$\delta^{\text{codim} \, \Ss}$. 

Of course, the probability of being within $\delta$ of $\Ss$ also depends on the value of the density in the vicinity of $\Ss$. One might think that the density might be relatively small near $\Ss$ so that getting data near $\Ss$ will be quite rare. On the contrary, in remark \ref{R:dnsity.contrs.as.data.spaces} it is argued that, at least in plane-fitting (chapter \ref{Chptr:sings.in.plane.fit}), the probability density can be arbitrarily high, in relative terms, near $\Ss$.

This raises some interesting practical questions. Suppose a data set $x$ is close to the singular set. Let $x'$ be an independent data set drawn from the same distribution. What is the approximate probability that $x'$ is also close to the singular set? Answering this question seems to require a Bayesian approach Berger \cite{joB85.BergerBayes}.
Might low codimension of $\Ss$ make it more ``likely'' that $x'$ is close to $\Ss$? Let $\Delta > 0$ and let $\pi^{\Delta}$ be the probability that $x'$ is more than $\Delta$ units distant from $\Ss$.
The possibility of using $x$ to construct a confidence set for $\pi^{\Delta}$ is discussed in appendix \ref{Chptr:misc.proofs}.

\section{Derivatives of data maps near singularities}  \label{SS:derivs.near.sings}
Near a singularity a small change in the data can make a big change in the value of the data map. In light of remark \ref{R:.completeness.of.F}, \emph{assume $\F$ is a complete metric space.}
If the co-dimension of the singular set, $\Ss$, is no larger then 1, then a ``big change'' 
in $\Phi(x)$, $x \in \D$, may just be a jump. However, if $\text{codim} \, \Ss > 1$, which will often be the case 
(e.g., figure \ref{F:LS.PC.LAD.lf.plots}(LS,PC)), then something more interesting happens. 

Suppose 
    \begin{multline}  \label{E:D.and.F.are.Riem.manifs}
      \D \text{ and } \F \text{ are Riemannian manifolds with Riemannian metrics }
        \langle \cdot, \cdot \rangle_{\D} \text{ and } \langle \cdot, \cdot \rangle_{\F}, \\
          \text{ respectively (resp.)}
    \end{multline} 
Let $\| \cdot \|_{\D, x}$ be the norm corresponding to 
$\langle \cdot, \cdot \rangle_{\D,x}$, $x \in \D$. Thus,
	\[
		\| v \|_{\D,x} = \sqrt{ \langle v, v \rangle_{\D,x} }, \quad v \in T_{x} (\D), \quad x \in \D.
	\]
Define $\| \cdot \|_{\F, w}$, the norm on $\F$ at $w \in \F$ similarly. 

We will assume further that \emph{the singular set $\Ss$ is closed, hence, locally compact.} (Chapter \ref{Chptr:severity} makes this plausible.)  Let $d = \dim \D$. Let $x_{0} \in \D \setminus \Ss$ and let $\varphi : \clU \to \RR^{d}$ be a coordinate neighborhood of $x_{0}$. We may assume $\clU \subset \D \setminus \Ss$. Let $G := \varphi(\clU) \subset \RR^{d}$ 
and let $\psi : G \to \clU$ be the inverse of $\varphi$. Then $\varphi$ and $\psi$ are locally Lipschitz by lemma \ref{L:coord.maps.are.Lip}\eqref{I:phi.psi.Lip}. Let $T(\D)$ and $T(\F)$ be the tangent bundles of $\D$ and $\F$, resp. 
Let $\pi_{\D} : T(\D) \to \D$ be the projection that takes $v \in T_{x}(\D)$ to $x \in \D$. 
Let $\tilde{\clU} := \pi_{\D}^{-1} ( \clU ) \subset T(\D)$. Note the difference: $\clU \subset \D$ but $\tilde{\clU} \subset T(\D)$. Thus, the differentials $\varphi_{\ast} : \tilde{\clU} \to G \times \RR^{d}$ 
and $\psi_{\ast} : G \times \RR^{d} \to \tilde{\clU} \subset \RR^{N}$ are smooth (Boothby \cite[Exercise 6, p.\ 337]{wmB75} and Milnor and Stasheff \cite[pp.\ 8--9]{jwMjdS74}).

In accordance with \eqref{E:D'.=.D.less.S}, $\Phi$ is defined and continuous 
on $\D \setminus \Ss$. WLOG (Without Loss Of Generality) we may assume that $\Phi(\clU)$ is a subset of a coordinate neighborhood $W \subset \F$ with coordinate map, $\tau : W \to \RR^{f}$, $f$ being the dimension of $\F$. Let $\sigma : H := \tau(W) \to \F$ be the inverse of $\tau$. Let $\tilde{W} := \pi_{\F}^{-1}(W) \subset T(\F)$, where $\pi_{\F} : T(\F) \to \F$ is projection. Thus, $\tau_{\ast} : \tilde{W} \to H \times \RR^{f}$ and $\sigma_{\ast} : H \times \RR^{f} \to \tilde{W} $ are smooth.

Consider the map $\Omega : G \to \RR^{f}$ defined by $\Omega = \tau \circ \Phi \circ \psi$. So $\Omega$ is continuous on $G$. Recall (Federer \cite[Section 3.1, p.\ 209]{hF69}) that, at each $y \in G$ at which it is defined, the differential, $D \Omega(y)$, is a linear function from $\RR^{d} \to \RR^{f}$.
By Federer \cite[p.\ 211]{hF69}, $D \Omega$ is Borel measurable on the set where it is defined. 
Moreover, the set where $D\Omega$ defined is itself Borel. Define, in the obvious manner, the composite 
$D \Phi := \sigma_{\ast} \circ D \Omega \circ \varphi_{\ast}$, which takes values in $\tilde{W}$, wherever it is defined on $\tilde{\clU}$. Let $x \in \clU$. If it is defined, denote the restriction, $D\Phi \restriction_{T_{x}(\D)}$ of $D\Phi$ to the tangent space $T_{x}(\D)$ to $\D$ at $x$ 
by $\Phi_{\ast,x}$ or just by $\Phi_{\ast}$, when $x$ is understood. If $\Phi_{\ast,x}$ is defined at $x \in \clU$ then, by (Federer \cite[(1) and (2), pp.\ 209--210]{hF69}), $\Phi_{\ast,x}$ is independent of the particular coordinate neighborhoods $(\clU, \phi)$ and $(W, \tau)$. And $D \Phi$ is Borel measurable off a Borel measurable set. We assume
	\begin{equation}   \label{E:Phi.star.defined.a.e.}
		\Phi_{\ast,x} \text{ is defined for } \Hm^{d} \text{-almost all } x \in \clU.
	\end{equation}

Thus, at $\Hm^{d}$-almost all $x \in \clU$ the data map $\Phi$ induces a linear operator 
$\Phi_{\ast} := \Phi_{\ast, x}: T_{x} (\D') \to T_{\Phi(x)} (\F)$. As such, 
$\Phi_{\ast}$ has a norm at each $x \in \clU$, where it is defined, \emph{viz.}:
	\begin{equation}   \label{E:operator.norm.of.Phi.ast}
		|\Phi_{\ast,x}| 
			:= \sup \Bigl\{ \bigl\| \Phi_{\ast}(v) \bigr\|_{\F, \Phi(x)} : 
			           v \in T_{x} (\D), \; \| v \|_{\D,x} = 1 \Bigr\}.
	\end{equation}
$| \Phi_{\ast}  |$ measures how big the derivative of $\Phi$ is at each point of $\clU$. 

Let $\alpha : [0, \lambda] \to \clU$ be a continuous, piece-wise differentiable curve in $\clU$. This means that there are finitely many values $0 = \lambda_{0} < \lambda_{1} < \cdots < \lambda_{m} = \lambda$ s.t.\ $\alpha$ is differentiable on each interval $(\lambda_{i-1}, \lambda_{i})$ ($i = 1, \ldots, m$). If $\alpha$ is differentiable at $t \in (0, \lambda)$ 
let $\alpha'(t) := \alpha_{\ast}(t): \RR \to T_{\alpha(t)} (\D)$ denote the differential of $\alpha$ at $t$. By  
Boothby \cite[p.\ 185; Theorem (1.2), p.\ 107; and Theorem (1.6), p.\ 109]{wmB75} and the ``area formula'', Federer \cite[3.2.3, p.\ 243]{hF69}, the length of $\alpha$ is 
       \begin{equation}  \label{E:alpha.length.integral}
		\text{length of } \alpha = \Hm^{1} \bigl( \alpha[0,\lambda] \bigr) 
		  = \int_{0}^{\lambda} \bigl\| \alpha'(t) \bigr\|_{\D, \alpha(t)} \, dt .
	\end{equation}
A similar formula applies to other curves. 

Assume $\Phi_{\ast, \alpha(s)}$ is defined for almost all $s \in [0, \lambda]$. Then it makes sense to define 
	\begin{equation}  \label{E:avg.size.Phi.alng.curve}
		\text{average size of the derivative of $\Phi$ along $\alpha$ } 
		       = \frac{1}{\text{length of } \alpha} 
		     \int_{\alpha[0, \lambda]}  | \Phi_{\ast,x}  | \, \Hm^{1}(dx).  
	\end{equation}

Proposition \ref{P:deriv.blows.up} below concerns the average size of the derivative 
of $\Phi$ along curves $\alpha$. 
In the proof, we use the notion of geodesic convexity (Boothby \cite[p.\ 337]{wmB75}, Spivak \cite[p.\ 491]{mS79.SpivakVol1}). We discuss that notion now before resuming our tale of $\alpha$, etc. 

Let $M$ be an $m$-dimensional Riemannian manifold. 

  \begin{definition}  \label{D:geodesic.convexity}
A subset $S$ of $M$ is ``geodesically convex'' if for every $x,y \in S$ there is a unique shortest geodesic path joining $x$ and $y$ and that path lies entirely in $S$.
  \end{definition}

Let $\phi$ be the topological metric on $M$ determined by its Riemannian metric 
(Boothby \cite[Theorem (3.1), p.\ 187]{wmB75}). 
By Helgason \cite[Theorem 9.9, p.\ 53]{sH62.SymSpaces} (and the remark preceding the theorem) we have the following.

  \begin{prop}  \label{P:geod.cnvx.nbhds.exist}
For every point $x_{0}$ of $M$ there is an $r(x_{0}) > 0$ s.t.:
   \begin{enumerate}
      \item For every $t \in (0, r(x_{0})]$ the neighborhood $B_{t}(x_{0}) \subset M$ 
      (the open ball about $x_{0}$ with radius 
      $t$ w.r.t.\ $\phi$; \ref{E:ball.defn}) is a normal neighborhood of each of its points 
      (Helgason \cite[p.\ 33]{sH62.SymSpaces}, 
        Boothby \cite[Definition (6.7), p.\ 335]{wmB75}).
      		\label{I:normal.nbhd}
      \item Let $t \in (0, r(x_{0})]$. 
      Let $x_{1}$, $x_{2} \in  B_{t}(x_{0})$ and $s = \phi (x_{1}, x_{2})$. Then there exists a geodesic
      $\omega_{x_{1} x_{2}} : [0,s] \to M$ of length $s$ joining $x_{1}$ and $x_{2}$. 
      (If $x \in B_{t}(x_{0})$ define $\omega_{x x} : \{ 0 \} \to \{ x \}$.)  In fact, 
      $\omega_{x_{1} x_{2}}$ is the only curve segment in $M$ of length at most $s$  
      joining  $x_{1}$ and $x_{2}$. Moreover, 
      $\omega_{x_{1} x_{2}}[0,s] \subset B_{t}(x_{0})$, i.e. $B_{t}(x_{0})$ is geodesically convex.  \label{I:geodesically.convex}
   \end{enumerate}
  \end{prop}

 \begin{remark}[Short geodesics in geodesically convex sets] \label{R:short.geods.geod.cnvx}
Let $M$ be as above, let $x_{0} \in M$ and let $r(x_{0}) > 0$ be as in proposition \ref{P:geod.cnvx.nbhds.exist}. Let $t \in (0, r(x_{0})]$ and $x \in B_{t}(x_{0})$. Then, by proposition \ref{P:geod.cnvx.nbhds.exist}\eqref{I:normal.nbhd}, there exists a star-like neighborhood $G_{x} \subset T_{x} M$ of $0_{x} \in T_{x} M$
s.t.\ $Exp$ is a bijection of $G_{x}$ onto $B_{t}(x_{0})$. Let $y \in B_{t}(x_{0})$ and let $X_{x} \in G_{x}$ satisfy $Exp (X_{x}) = y$. 

Let $\lambda$ be a geodesic arc joining $x$ to $y$ s.t.\ the length of $\lambda$ is $\phi(x,y)$. By part \ref{I:geodesically.convex} of the proposition, $\lambda$ is unique and its image lies in $B_{t}(x_{0})$. 
Therefore, by Helgason \cite[Proposition 5.3, p.\ 30]{sH62.SymSpaces} and Boothby \cite[Lemma (5.21), p.\ 327]{wmB75}, up to an affine change of variable, $\lambda(t) = Exp \bigl[ t \lambda'(0) \bigr]$. But $Exp$ is one-to-one on $G_{x}$. Therefore, $\lambda'(0) \propto X_{x}$. 
Thus, $\lambda$ and $t \mapsto Exp (t X_{x})$ ($0 \leq t \leq1$) parametrize the same curve in $B_{t}(x_{0})$ and, by Boothby \cite[p.\ 333]{wmB75}, 
$\phi(x,y) = \text{length of } \lambda = |X_{x}|$. 
 \end{remark} 
 
Now back to $\alpha$, etc. Let $\delta$ be a topological metric on $\D$ and 
let $\delta_{\langle \cdot, \cdot \rangle}$ be the topological metric on $\D$ generated by its Riemannian metric, $\langle \cdot, \cdot \rangle_{\D}$. For $\eta > 0$, let $\mcl{B}_{\eta}(x)$ be the open ball in $\D$ about $x$ with radius $\eta_{0}$, defined using the either metric $\delta$ 
or $\delta_{\langle \cdot, \cdot \rangle}$. Let $\rho$ be the topological metric on $\F$ generated by its Riemannian metric, $\langle \cdot, \cdot \rangle_{\F}$. 

    \begin{prop}  \label{P:deriv.blows.up}
Suppose $\F$ is a manifold that is complete w.r.t.\ $\rho$. Suppose the metrics $\delta$ and $\delta_{\langle \cdot, \cdot \rangle}$ on $\D$ are locally equivalent in the sense that $\D$ is covered by open sets $\clU$ with the following property. There exists $K(\clU) \in [1, \infty)$ (depending on $\clU$) s.t.\ for every $x,y \in \clU$ we have
	\begin{equation}  \label{E:two.deltas.locally.equivalent}
	     K(\clU)^{-1} \delta(x,y) \leq\delta_{\langle \cdot, \cdot \rangle}(x,y) 
		  \leq K(\clU) \,\delta(x,y).
	\end{equation} 

Suppose $\Ss$ is locally compact with $\text{codim} \, \Ss > 1$. (In particular, 
$d := \dim \D > 1$.) Assume \eqref{E:Phi.star.defined.a.e.} holds. Then there exists a functions $C : \Ss \to (0, 1)$ and 
$\eta_{0} : \Ss \to (0, \infty)$ with the following property. Let $x \in \Ss$ and let 
$\eta \in \bigl( 0, \eta_{0}(x) \bigr)$.  Then there exists $\lambda \in (0, \infty)$ and a path 
$\alpha : [0, \lambda] \to \mcl{B}_{\eta}(x) \setminus \Ss$
s.t.\ $\alpha$ is one-to-one, and
  \begin{enumerate}
    \item $\Phi_{\ast, \alpha(s)}$ is defined for almost all $s \in [0,\lambda]$ and the average size of the derivative of $\Phi$ along $\alpha$ (as defined by \eqref{E:avg.size.Phi.alng.curve}) is greater than $C(x) / \eta$.  
    	\label{I:avg.size.bggr.C/eta}
    \item The average distance from $\alpha$ to $x$ satisfies
	    \[
		    C(x) \eta 
		       \leq \frac{1}{\text{length of } \alpha} \int_{\alpha[0, \lambda]}  
		         \delta(y,x) \, \Hm^{1}(dy) \leq \eta.
	    \]
	    \label{I:avg.dist.bggr.C.eta}
  \end{enumerate}
   \end{prop}
(For proof see appendix \ref{Chptr:misc.proofs}.)

Part \eqref{I:avg.dist.bggr.C.eta} of the proposition shows that the bound in \eqref{I:avg.size.bggr.C/eta} does not depend on $\alpha$ lying arbitrarily close to $x$. 
Let $\overline{dist} := \text{average distance from } \alpha \text{ to } x$. Then \ref{I:avg.dist.bggr.C.eta} implies $1/\bigl( C(x) \eta  \bigr) \geq 1/\overline{dist}$. Multiply both sides by $C(x)^{2}$. The proposition then implies  
	\begin{equation}  \label{E:avg.deriv.avg.dist}
		\text{average size of the derivative of $\Phi$ along } \alpha
		  \geq \frac{C(x)}{\eta} 
		    \geq \frac{C(x)^2}{\text{average distance from } \alpha \text{ to } x}.
	\end{equation}
Thus, in an ``average'' sense, the derivative blows up at least as fast as the reciprocal of the distance to $x$. But this assertion with the adjective``average'' removed is false as shown in the following example. However, for statistical purposes, perhaps the ``average'' behavior is sufficient. 

  \begin{example} \label{Ex:alpha.arb.close.deriv.arb.large}
\eqref{E:avg.deriv.avg.dist} seems to say that that $|\Phi_{\ast}|$ increases as the reciprocal of the distance to the singularity, $x$. Here, however, is an example in which $|\Phi_{\ast}|$ fails to increase that fast. Let $\F = \RR$ and to start with let $\D := [0,1]$. For $t \in (0,1]$, let $f(t) := \log \Bigl[ -\log \bigl(  t/e \bigr) \Bigr]$, where $e := \exp(1)$. So $f(1) = 0$ and $f(t) \uparrow \infty$ as $t \downarrow 0$. For $n = 0, 1, 2, \ldots$, let 
    \begin{equation*}
      t_{n} = f^{-1}(n) , 
    \end{equation*}
More precisely, $t_{n} = \exp \{ 1 - e^{n} \}$. So $t_{n} \downarrow 0$ 
as $n \to \infty$. Therefore, $t_{n+1} \ll t_{n}$. 

Now let 
	\begin{equation*}
		g(t) := 
			\begin{cases}
				f(t) - f \bigl( t_{n} \bigr) = f(t) - n, 
				    &\text{if } t_{n+1} \leq t < t_{n} 
				      \text{ for some even } n \geq 0, \\
				f \bigl( t_{n+1} \bigr) - f(t) = n - f(t) + 1, 
				    &\text{if } t_{n+1} \leq t < t_{n} \text{ for some odd } n > 0.
			\end{cases}
	\end{equation*}
Then $g : (0, 1] \to [0,1]$ and $g(t)$ oscillates between 0 and 1 as $t \downarrow 0$. $g$ is continuous. So, with $\D := [0,1]$,  0 is the unique singularity of $g$. 
$g$ is also differentiable on $(0,1)$ except at $t_{0}, t_{1}, \ldots$. If $g$ is differentiable at $t \in (0,1]$, then
	\[
	     \bigl| g'(t) \bigr| = \frac{1}{ t \bigl| \log(t/e) \bigr| } = \bigl| f'(t) \bigr| = - f'(t).
	\]
Hence, $\bigl| g'(t) \bigr| \to \infty$ more slowly than $1/t$ as $t \downarrow 0$. 

Let $m < n$ be positive integers and consider the ``curve'', $\alpha = \alpha_{m,n}$, that is just the identity on $[t_{n}, t_{m}]$. 
Thus, $\alpha[t_{n}, t_{m}] = [t_{n}, t_{m}] \subset \mcl{B}_{\eta}(0) \subset \RR$, 
with $\eta := 2 t_{m}$, say. The average size, $\overline{|g'|}$, of $g'$ over $\alpha$ is
    \begin{multline*}
      \overline{|g'|} = (t_{m} - t_{n})^{-1} \int_{t_{n}}^{t_{m}} |g'(t)| \, dt \\
         = -(t_{m} - t_{n})^{-1} \sum_{k=m+1}^{n} \int_{t_{k}}^{t_{k-1}} f'(t) \, dt 
          = (t_{m} - t_{n})^{-1} \sum_{k=m+1}^{n} \bigl( f(t_{k}) - f(t_{k-1}) \bigr) \\
            = (t_{m} - t_{n})^{-1} \sum_{k=m+1}^{n} \bigl( k - (k-1)) \bigr)
             = (t_{m} - t_{n})^{-1} (n-m) .
    \end{multline*}
Now, $t_{m} >  t_{n}$ and $n-m \geq 1$. Therefore, $1 < (n-m)/(1 - t_{n}/t_{m})$. It follows that $(t_{m} - t_{n})^{-1} (n-m) > 1/t_{m}$. Let $C(0) := 1/4$. Then 
    \begin{equation*}
      \overline{|g'|} = \frac{n-m}{t_{m} - t_{n}} > \frac{1}{t_{m}} > \frac{1/4}{2 t_{m}}
        = \frac{C(0)}{\eta} ,
    \end{equation*}
as required by part \eqref{I:avg.size.bggr.C/eta} of proposition \ref{P:deriv.blows.up}.

The average distance, $\overline{dist}$, from $\alpha$ to $\Ss = \{0\}$ is 
$(t_{m} + t_{n})/2$, so $\overline{dist} \to 0$ as $m, n \to \infty$. But
    \begin{equation*}
       C(0) \eta = 2 t_{m} / 4 < (t_{m} + t_{n})/2 < 2 t_{m} = \eta .
    \end{equation*}
So part \eqref{I:avg.dist.bggr.C.eta} of the proposition holds.

Therefore, by \eqref{E:avg.deriv.avg.dist}, $\overline{|g'|}$ increases at least as fast 
as $\overline{dist}$. That is despite the fact that 
$\bigl| g'(t) \bigr| \to \infty$ more slowly than $1/t$ as $t \downarrow 0$.

But with $\D := [0,1]$ this is not quite the framework of \eqref{E:D.and.F.are.Riem.manifs} and proposition \ref{P:deriv.blows.up}. So let $d > 1$ and
let $\Phi : \D := B_{1}^{d}(0) \to [0,1]$ (see \eqref{E:Euc.ball.defn}) be defined by 
	\[
		\Phi(x) = g \bigl( |x| \bigr) .
	\]
Thus, the singular set, $\Ss$, of $\Phi$ is still the origin. By \eqref{E:operator.norm.of.Phi.ast}, 
$|\Phi_{\ast,x}| = \bigl| g' \bigl( |x| \bigr) \bigr|$, wherever it is defined. In particular, the size of the derivative of $\Phi(y)$ increases more slowly than $1/|y| = 1/dist(y, \Ss)$ as $y \to \Ss$.The preceding argument goes through with 
$\alpha(t) := (t, 0, \ldots, 0) \in B_{1}^{d}(0)$, ($t \in [t_{n}, t_{m}]$). 
	   \end{example}

Finally, we have a simple inequality that we find frequent use for.
By the (Cauchy-)Schwarz inequality (Stoll and Wong \cite[Theorem 3.1, p.\ 79]{rrSetW68.LinearAlgebra}),
	\begin{multline}  \label{E:n.c.sqrd.sum.ineq}
	    n^{2} \bigl( \max_{i} |c_{i}| \bigr)^{2} \geq
	     n \sum_{i=1}^{n} c_{i}^{2} \geq \left( \sum_{i=1}^{n} |c_{i}| \right)^{2}
	       \geq \sum_{i=1}^{n} c_{i}^{2} \geq \bigl( \max_{i} |c_{i}| \bigr)^{2}, \\
	         \quad n = 1, 2, \ldots; \; c_{1}, \ldots, c_{n} \in \RR .
	\end{multline}
   
\chapter{Topology} \label{Chptr:topology} 
This chapter uses algebraic topology to derive basic results concerning singularity. Some topological interludes are present in later chapters, too. 
\section{Homology and Singularity} \label{SS:homol.and.sing}
The following is a main theorem of the book in the sense almost everything else in this book is a consequence of it. (Theorems \ref{T:lwr.bnd.on.Haus.meas} and \ref{T:if.lin.combo.on.F.then.can.rstrct.to.bad.sings}  are the other main theorems of the book. Their usefulness derives from theorem \ref{T:Phi.star.Hr.contains.Theta.star.Hr}.) The theorem shows that the behavior of $\Phi$ near an appropriately chosen ``test pattern space'', 
$\T \subset \D$, can have implications for the global stability of $\Phi$ on $\D$. Let $\Ss$ be the singular set of $\Phi$ (chapter \ref{Chptr:basic.setup}). By \eqref{E:D'.=.D.less.S}, we assume that 
$\D' := \D \setminus \Ss$ is dense and $\Phi$ is defined and continuous on $\D'$ 
with codomain $\F$. Hence, if $\Ss' \supset \Ss$ then the data map $\Phi$ is continuous 
on $\D \setminus \Ss'$. 
Suppose $\Ss' \cap \T = \varnothing$ and let $k : \T \hookrightarrow \D \setminus \Ss'$ be inclusion. 
Then $\Theta := \Phi \circ k : \T \to \F$. Hence,  
    \begin{equation*}
      \Phi_{\ast} \bigl[ H_{r}(\D  \setminus \Ss') \bigr]
        \supset \Theta_{\ast} \bigl[ H_{r}(\T) .
    \end{equation*}

The following describes circumstances in which essentially the same thing holds even if 
$\Ss' \cap \T \neq \varnothing$ providing that $\Ss' \cap \T$ is small. Recall that for 
$s \geq 0$, 
$\Hm^{s}$ denotes $s$-dimensional Hausdorff measure. (See \eqref{E:Hs.defn},)
   \begin{theorem}  \label{T:Phi.star.Hr.contains.Theta.star.Hr}  
       IF: 
        \begin{enumerate}
           \item $\T \subset \D$ is a compact $t$-dimensional manifold 
             (in the relative topology).   \label{Hyp:T.manif}
           \item $\Ss' \subset \D$ has empty interior, $\Ss' \cap \T$ is closed, 
             and $\Phi$ is continuous on $\D \setminus \Ss'$.  
                        \label{Hyp:S.cap.T.closed}
           \item $r$ is an integer between 0 and $t$, inclusive.  \label{Hyp:r.integer}
           \item $\Hm^{t-r}(\Ss' \cap \T) = 0$, e.g., $\dim ( \Ss' \cap \T ) < t - r$ 
           (e.g., $\Ss' \cap \T = \varnothing$), so $\T \setminus \Ss'$ is dense in $\T$.  
                \label{Hyp:S.cap.T.small}
           \item The restriction of $\Phi$ to $\T \; \setminus \; \Ss'$ has a, necessarily 
           unique, continuous extension, $\Theta$, to $\T$. 
               \label{Hyp:extend}
         \end{enumerate}
      THEN:
         \begin{equation} \label{E:Phi.ast.image.includes.Theta.ast.image}
            \Phi_{\ast} \bigl[ H_{r}(\D  \setminus \Ss') \bigr]
                \supset \Theta_{\ast} \bigl[ H_{r}(\T) \bigr]
                  \text{ as subgroups of } H_{r}(\F) .
         \end{equation}
      (``$H_{r}$'' indicates $r$-dimensional singular homology, 
Munkres \cite[Chapter 4 and pp.\ 309--310]{jrM84}. If $\T$ is not orientable, use a commutative ring of characteristic 2 for coefficients. Otherwise, any commutative coefficient ring is permissible.)
   \end{theorem}

\emph{Note: In this book all rings are assumed commutative with unity element.}

Let $\Ss$ be the singular set of $\Phi$. With a view to satisfying \textbf{hypothesis \ref{Hyp:S.cap.T.closed}} of the theorem, we might define 
$\Ss' := \Ss \cup (\overline{\Ss \cap \T})$, where $\overline{E}$ indicates closure of a set 
$E \subset \D$. So $\Ss' \cap \T$ is closed.    

   \begin{remark}[Singularities on and near $\T$] \label{R:Phi.on.and.near.T}
By \eqref{E:Haus.measure.of.empty.set.is.0}, if $\Ss' \cap \T = \varnothing$ then 
$\Hm^{t-r}(\Ss' \cap \T) = 0$ no matter what $t-r$ is. 
At first blush it might seem that \textbf{hypothesis \ref{Hyp:extend}} of the theorem automatically implies \textbf{hypothesis \ref{Hyp:S.cap.T.small}}. But \textbf{hypothesis \ref{Hyp:extend}} only pertains to the behavior of the \emph{restriction} $\Phi \restriction_{\T \setminus \Ss'}$ of $\Phi$ 
to $\T \setminus \Ss'$. In the theorem that restriction is not allowed to have singularities. In section \ref{SS:approx.cont.} we discuss what can be said when \textbf{hypothesis \ref{Hyp:extend}} is dropped. \textbf{Hypothesis \ref{Hyp:S.cap.T.small}} takes into account the behavior of $\Phi$ in an arbitrary neighborhood of $\T$. 
   \end{remark}

The assumption that $\Hm^{t-r}(\Ss' \cap \T) = 0$ can be replaced by the weaker assumption $\check{H}^{t-r}(\T \cap \Ss') = 0$, where ``$\check{\;\,}$'' indicates 
\v{C}ech cohomology (Dold \cite[Chapter VIII, chapter 6]{aD95.alg.topol}). (See \eqref{E:cohom.and.codim}.)

Choosing $\T$ so that $\T \cap \Ss'$ is not too big (\textbf{hypothesis \ref{Hyp:S.cap.T.small}}) is where one tends to get into trouble in trying to apply this result. Note that since $\T$ is compact and $\D$ is a metric space (hence Hausdorff), $\Ss' \cap \T$ is closed in $\T$ if and only if it is closed in $\D$ (Simmons \cite[Theorem D, p.\ 131]{gfS63}). Note further that in \textbf{hypothesis \ref{Hyp:T.manif}}, $\T$ may be a finite set (a compact 0-dimensional manifold). In section \ref{SS:gnrl.lwr.bnd.plane.fit} we explore, in the context of plane-fitting, what happens when the assumption that $\T$ is compact is dropped.

Note further that the hypotheses of theorem \ref{T:Phi.star.Hr.contains.Theta.star.Hr} can be checked by only examining the behavior of $\Phi$ in and in the immediate vicinity of $\T$. (This is an example of the ``Sales Pitch'', remark \ref{R:sales.pitch}.)  
For theorem \ref{T:Phi.star.Hr.contains.Theta.star.Hr} we only require $\Ss' \cap \T$ to be closed. But in other settings we need $\Ss'$ itself to be closed (e.g., in corollary \ref{C:consequences.of.H.r.D.=.0}, proposition \ref{P:sing.dim.when.H.d-r.D.=.0}, and theorem \ref{T:lwr.bnd.on.Haus.meas}; see remark \ref{R:Ss.closed?}). That assumption cannot be checked by only looking in the vicinity of $\T$. Often the ``severity trick'', chapter \ref{Chptr:severity}, comes to the rescue.

  \begin{remark}[Complexity and Perfection]  \label{R:complexity}
Perfection is an absolute concept, but sometimes there are degrees of perfection of fit. This sort of thing can happen in hypothesis testing (example \ref{Ex:hypothesis.testing}, more generally example \ref{Ex:disconnected.F}), in plane-fitting (remark \ref{R:degenerate.data}) or in linear classification (remark \eqref{R:Pf.in.lin.class}).   

One might generalize the notion of perfect fits (and hence of test patterns) 
by replacing $\Pf$ by a function 
$\gamma : \D \to \RR$, specific to the class of data maps under consideration, that measures what one might call the ``complexity'' of data sets, where ``complexity'' is the opposite of ``perfection''. Such an approach might fit in with methods involving regularization (remark \ref{R:regularization.generalities}). The smaller $\gamma(x)$ 
($x \in \D'$) is the more ``perfect'' $x$ is and the more tightly constrained is $\Phi(x)$. Let
    \begin{equation} \label{E:indicator.fn.defn}
      1^{S} \text{ denote the indicator or characteristic function of a set } S.
    \end{equation}
If a space $\Pf$ of perfect fits can be specified then the corresponding 
$\gamma$ is $1^{\Pf^{c}}$, the indicator (or characteristic) function of the complement of 
$\Pf$: 
Perhaps in this looser framework one might get results like those in this chapter by applying persistent homology (Edelsbrunner and Harer \cite[Part C]{hEjlH10.CompTopol}) to $P$. In this book we do not explore that possibility. 
  \end{remark}

\begin{remark}[``Stationary'' data sets]
In settings where derivatives of $\Phi$ make sense, it is natural to wonder whether the preceding result (and other results in this book) might have something interesting to say about those derivatives. For example, can this approach lead to results similar to proposition \ref{P:deriv.blows.up}?

A data map, $\Phi$, should not be too sensitive to small changes in the data. But other extreme is also bad: $\Phi$ should be sensitive to \emph{some} changes in the data. In particular, it might be bad if $\Phi$ were stationary at some data set, $x$, i.e., if $d \Phi(x) = 0$, where $d \Phi$ is the differential of $\Phi$. This is a form of singularity that might be studied using the methods of this book. Just replace the feature space $\F$ by the tangent bundle $T \F$ with the 0-vectors removed. However, the singularities of $\Phi$ are also singularities of $d \Phi$ and it might be difficult to distinguish the two forms of singularity. 
\end{remark}

Think of $\Theta$ as the restriction, $\Sigma \restriction_{\T}$ (see \eqref{E:restriction.notation}), of the standard, 
$\Sigma$, to $\T$. (See subsection \ref{SS:calibration}.) Theorem \ref{T:Phi.star.Hr.contains.Theta.star.Hr} is only interesting if
$\Theta_{\ast} \bigl[ H_{r}(\T) \bigr] \subset H_{r} (\F)$ is  nontrivial when $r > 0$. 
As for the case $r = 0$, if $\Phi_{\ast} \bigl[ H_{0}(\D  \setminus \Ss') \bigr]$ is isomorphic to the coefficient ring, that says nothing beyond that  
$H_{0}(\D  \setminus \Ss')$ is non-empty. In summary, the theorem is only interesting when
   \begin{equation} \label{E:nontriv.r-dim.homol}
      \Theta_{\ast} \colon \tilde{H}_{r}(\T) \rightarrow \tilde{H}_{r} (\F) 
        \text{ is nontrivial.}
   \end{equation}
where ``$\, \tilde{\;} \,$'' indicates reduced homology. In applying theorem \ref{T:Phi.star.Hr.contains.Theta.star.Hr} the idea is to choose $\T$ so the hypotheses of Theorem \ref{T:Phi.star.Hr.contains.Theta.star.Hr} and \eqref{E:nontriv.r-dim.homol} are true and easy to check.

  \begin{remark}[$r = 0$]  \label{R:nontriv.Theta.star.in.dim.0}
Here we analyze the $r = 0$ case. This work is continued in example \ref{Ex:disconnected.F}. Suppose \eqref{E:nontriv.r-dim.homol} holds and so do the hypotheses of theorem \ref{T:Phi.star.Hr.contains.Theta.star.Hr}. Assume $\T$ is disconnected. Use integer coefficients for homology. By Munkres \cite[Theorem 29.4, p.\ 164]{jrM84}, 
we already know $H_{0}(\D  \setminus \Ss')$ is a nontrivial, even if $\Ss'$ is empty. We do not need \eqref{E:nontriv.r-dim.homol} and \eqref{E:Phi.ast.image.includes.Theta.ast.image} to tell us that. 

Let $\Theta$ be as in theorem \ref{T:Phi.star.Hr.contains.Theta.star.Hr} \textbf{hypothesis \ref{Hyp:extend}}. By Munkres \cite[Theorem 29.4, p.\ 164]{jrM84} again, $H_{0}(\D  \setminus \Ss')$, $H_{0}(\T)$, and $H_{0}(\F)$ are all free abelian. It follows from Munkres \cite[Theorem 4.2, p.\ 24]{jrM84} that both 
$\Phi_{\ast} \bigl[ H_{0}(\D  \setminus \Ss') \bigr]$ and $\Theta_{\ast} \bigl[ H_{0}(\T)$ are also free abelian. 

Here we show that \eqref{E:nontriv.r-dim.homol} implies 
    \begin{equation}  \label{E:H0(T).has.rank.>.1}
      \Theta_{\ast} \bigl( H_{0}(\T) \bigr) \text{ has rank } > 1 , 
    \end{equation}
which means, by \eqref{E:Phi.ast.image.includes.Theta.ast.image}, that $\Phi_{\ast} \bigl[ H_{0}(\D  \setminus \Ss') \bigr]$ has rank $> 1$, which in turn implies
    \begin{equation}  \label{E:H0(T).has.rank.>.1}
      H_{0}(\D  \setminus \Ss') \text{ is free with rank } > 1 , 
    \end{equation}
which means, by Munkres \cite[Theorem 29.4, p.\ 164]{jrM84}, yet again
that $\D \setminus \Ss'$ is not path-connected (and neither is $\F$). This is applied in example \ref{Ex:disconnected.F}.

Let $x_{0} \in \T$ and let $\gamma_{\T} : \T \to x_{0}$ be the constant map. Similarly, 
let $\gamma_{\F} : \F \to y_{0} := \Theta(x_{0})$ be constant. 
By Dold \cite[Definition 4.3, pp.\ 33--34]{aD95.alg.topol}, the following commutes with exact rows.
     \begin{equation}  \label{E:T,F.dim.0.CD}
          \begin{CD}
          	0 @>>> \tilde{H}_{0}(\T) @>{i_{\T}}>> H_{0}(\T) 
	          @>{\gamma_{\T\ast}}>> H_{0}(x_{0}) @>>> 0 \\
	         & & @V{\Theta_{\ast}}VV  @V{\Theta_{\ast}}VV 
	           @V{\Theta_{\ast}}V{\isomto}V \\
	       0 @>>> \tilde{H}_{0}(\F)  @>>{i_{\F}}> H_{0}(\F) 
	         @>>{\gamma_{\F\ast}}> H_{0}(y_{0}) @>>> 0 .
          \end{CD}
    \end{equation}
Here $i_{\T}$ and $i_{\F}$ are inclusions of groups. 
Let $\alpha_{\T} \in H_{0}(\T)$ and $\alpha_{F} \in H_{0}(\F)$ be the homology classes represented by the singular 0-simplices $\Delta_{0} \mapsto x_{0} \in \T$ and 
$\Delta_{0} \mapsto y_{0} \in \F$, respectively. Here, $\Delta_{0}$ is the 0-simplex. 
Then $\Theta_{\ast}$ maps $\alpha_{\T}$ to $\alpha_{F} \neq 0$.

Similarly, let $\xi_{\T} \in H_{0}(x_{0}) \isomto \ZZ$ 
and $\xi_{F} \in H_{0}(y_{0}) \isomto \ZZ$ be the homology classes represented by 
$\Delta_{0} \mapsto x_{0} \in \{x_{0}\}$ and 
$\Delta_{0} \mapsto y_{0} \in \{y_{0}\}$, respectively. Then $\Theta_{\ast}$ maps 
$\xi_{\T}$ to $\xi_{F} \neq 0$. Moreover, 
$\gamma_{\T\ast}(\alpha_{\T}) = \xi_{\T}$ and 
$\gamma_{\F\ast}(\alpha_{\F}) = \xi_{\F}$. 

By Munkres \cite[Corollary 23.2, p.\ 132]{jrM84}, each row in \eqref{E:T,F.dim.0.CD} splits. Thus, $\tilde{H}_{0}(\T) \oplus H_{0}(x_{0}) \isomto H_{0}(\T)$ and 
$\tilde{H}_{0}(\F) \oplus H_{0}(y_{0}) \isomto H_{0}(\F)$. 
In fact, let $j_{\T} : H_{0}(x_{0}) \to H_{0}(\T)$ be generated by $\xi_{\T} \mapsto \alpha_{\T}$. So $\gamma_{\T\ast} \circ j_{\T}$ is the identify on $H_{0}(x_{0})$. Similarly, let $j_{\F} : H_{0}(y_{0}) \to H_{0}(\F)$ be generated by $\xi_{\F} \mapsto \alpha_{\F}$. So $\gamma_{\F\ast} \circ j_{\F}$ is the identify on $H_{0}(y_{0})$. The following commutes. 
     \begin{equation}  \label{E:x0.y0.dim.0.CD}
          \begin{CD}
          	\tilde{H}_{0}(\T) \oplus H_{0}(x_{0}) @>{i_{\T} \oplus j_{\T}}>{\isomto}> 
	              H_{0}(\T) \\
	         @V{\Theta_{\ast} \oplus \Theta_{\ast}}VV  @V{\Theta_{\ast}}VV \\  
	         \tilde{H}_{0}(\F) \oplus H_{0}(y_{0}) @>{i_{\F} \oplus j_{\F}}>{\isomto}> 
	              H_{0}(\F) .
          \end{CD}
    \end{equation} 

By Munkres \cite[Theorem 29.4, p.\ 164]{jrM84} still again, we have that  
$\tilde{H}_{0}(\F) \oplus H_{0}(y_{0})$ and $H_{0}(\F)$ are free. So, by Munkres \cite[Theorem 4.2, p.\ 24]{jrM84} again, 
$(\Theta_{\ast} \oplus \Theta_{\ast}) \bigl( \tilde{H}_{0}(\T) \oplus H_{0}(x_{0}) \bigr)$ and 
$\Theta_{\ast} \bigl( H_{0}(\T) \bigr)$ are free. By \eqref{E:x0.y0.dim.0.CD}, these groups are isomorphic. But by \eqref{E:nontriv.r-dim.homol}, the first 
$\Theta_{\ast}$ in $\Theta_{\ast} \oplus \Theta_{\ast}$ is nontrivial. The second is an isomorphism. Therefore, 
$(\Theta_{\ast} \oplus \Theta_{\ast}) \bigl( \tilde{H}_{0}(\T) 
\oplus H_{0}(x_{0}) \bigr)$ and, hence, \eqref{E:H0(T).has.rank.>.1} holds.
$\Theta_{\ast} \bigl( H_{0}(\T) \bigr)$ have rank $> 1$. 
Therefore, \eqref{E:H0(T).has.rank.>.1} holds and $\D \setminus \Ss'$ is not path-connected.
  \end{remark}

   \begin{example}
For \eqref{E:nontriv.r-dim.homol} to hold it must be the case that $\tilde{H}_{r} (\F)$ is nontrivial. Situations in which 
$H_{r} (\F) \neq \{ 0 \}$ are discussed in example \ref{Ex:disconnected.F} and chapters \ref{Chptr:sings.in.plane.fit}, \ref{Chptr:spherical.location}, \ref{Chptr:aug.direct.mean}, and \ref{Chptr:robst.loc.on.circle}.

Cases in which $H_{r} (\F) = \{ 0 \}$ include;
	\begin{enumerate}
	  \item $\Phi$ is linear, e.g.\ $\Phi$ is the sample mean or, more generally, 
	  the analysis of variance (Scheff\'{e} \cite{hS59.ANOVA}).
	  \item $\Phi$ is the covariance or correlation matrix 
	    (Johnson and Wichern \cite[pp.\ 11--12]{raJdwW92}). (The space of correlation matrices is starlike with respect to the identity matrix.)
	  \item Persistent homology of point clouds. (The space finite persistence diagrams with the ``bottleneck distance'', Edelsbrunner and Harer \cite[pp.\ 180--181]{hEjlH10.CompTopol}, is contractible to the persistence diagram consisting only of the diagonal.) It is satisfying to note that persistent homology is a stable operation on point clouds 
(Chazal \emph{et al} \cite[Theorem 3.1]{fCdC-SljGfMsyO2009.StabilityOfPersistence}, there seems to be a typo in the theorem). Related case: Hierarchical clustering  (Johnson and Wichern \cite[p.\ 586]{raJdwW92}). (The space of dendograms is contractible to the trivial dendogram consisting of a single point.)
	\end{enumerate}    
   \end{example}  
   
   \begin{remark}[$\Ss'$ when $H_{r} (\F) = \{ 0 \}$]
However, $H_{r} (\F) = \{ 0 \}$ does not automatically render theorem \ref{T:Phi.star.Hr.contains.Theta.star.Hr} inconsequential. Suppose $H_{r} (\F) = \{ 0 \}$ and $\F$ can be written $\F = \F_{1} \cup \F_{2}$. 
Suppose \ref{E:nontriv.r-dim.homol} and the hypotheses of the theorem hold with $\F$ replaced 
by $\F_{1}$, $\D$ replaced by $\D_{1} := \Phi^{-1}(\F_{1})$, and $\Phi$ replaced by the restriction, 
$\Phi \restriction_{\D_{1}}$. In particular, $H_{r} (\F_{1}) \neq \{ 0 \}$. Then $\Ss'$ is non-empty. 
If data in $\D_{1}$ are common, this is of practical import.
   \end{remark}

The following gathers together some basic facts.
   \begin{corly} \label{C:consequences.of.H.r.D.=.0}
 Suppose the hypotheses of Theorem \ref{T:Phi.star.Hr.contains.Theta.star.Hr} and \eqref{E:nontriv.r-dim.homol} hold. 
      \begin{enumerate}
      \item Let $inc : \T \hookrightarrow \D$ be inclusion. 
      If $\Theta_{\ast} : H_{r}(\T) \to H_{r}(\F)$ is nontrivial and injective 
      but, in dimension $r$, $inc_{\ast}$ is not injective. 
          Then $\Ss'$ is nonempty. \label{I:S.aint.empty}
      \item Like (\ref{I:S.aint.empty}) but no longer require 
      $\Theta_{\ast} : H_{r}(\T) \to H_{r}(\F)$ to be injective, only nontrivial. If 
      $H_{r}(\D) = \{ 0 \}$ then again $\Ss'$ is nonempty. \label{I:S.still.aint.empty}
      \item  Suppose $\Ss'$ is closed in $\D$, but 
       \begin{equation}  \label{E:inclusion.induces.triviality}
          \text{inclusion $\D \setminus \Ss' \hookrightarrow \D$ 
            induces a trivial homomorphism in $r$-dimensional homology.} 
       \end{equation}
      (E.g.\ $H_{r}(\D) = \{ 0 \}$.) If $\mcl{V}$ is any open neighborhood of $\Ss'$ then the restriction 
      of $\Phi$ to       $\mcl{V} \setminus \Ss'$ induces a nontrivial homomorphism 
      in $r$-dimensional homology.
         \label{I:Phi.of.U.less.S.has.non0.r.homol}
      \end{enumerate}
   \end{corly}
   
\begin{proof} 
(\ref{I:S.aint.empty}): If $\T \cap \Ss \neq \varnothing$ we are done. So assume 
$\T \cap \Ss = \varnothing$. Then $\Theta = \Phi \restriction_{\T}$. Let 
$i : \T \hookrightarrow \D \setminus \Ss$ and $j : \D \setminus \Ss \hookrightarrow \D$ be inclusions. The following commutes:
     \begin{equation}  \label{E:gdil.zeta.CD}
          \begin{CD}
          	H_{r}(\T) @= H_{r}(\T) @>{\Theta_{\ast}}>> H_{r}(\F) \\
	        @V{inc_{\ast}}VV  @V{i_{\ast}}VV  @AA{\Phi_{\ast}}A \\
	       H_{r}(\D) @<<{j_{\ast}}< H_{r}(\D \setminus \Ss) 
	         @=H_{r}(\D \setminus \Ss) .
          \end{CD}
    \end{equation}
There exists $\alpha \in H_{r}(\T)$ s.t.\ $\alpha \neq 0$ but $inc_{\ast}(\alpha) = 0$. 
Since $\Theta_{\ast}$ is injective, we have $\Theta_{\ast}(\alpha) \neq 0$. By \eqref{E:Phi.ast.image.includes.Theta.ast.image} there exists 
$\beta \in H_{r}(\D \setminus \Ss)$ s.t.\ 
$\Phi_{\ast}(\beta) = \Theta_{\ast}(\alpha)$. Therefore 
$\beta \neq 0$. But $inc_{\ast}(\alpha) = 0$. Therefore, by commutativity of the diagram, $j_{\ast}(\beta) = 0$. The only way this can happen is if $\Ss \neq \varnothing$.

(\ref{I:S.still.aint.empty}): Since $\Theta_{\ast}$ is nontrivial, there exists 
$\alpha \in H_{r}(\T)$ 
s.t.\ $\Theta_{\ast}(\alpha) \neq 0$. Automatically, $inc_{\ast}(\alpha) = 0$. The proof now proceeds like that of (\ref{I:S.aint.empty}).

(\ref{I:Phi.of.U.less.S.has.non0.r.homol}):  Let $\mcl{V} \subset \D$ be an open neighborhood of $\Ss'$ and let 
$i: \mcl{V} \setminus \Ss' \hookrightarrow \D \setminus \Ss'$, 
$j : (\mcl{V}, \mcl{V} \setminus \Ss') \hookrightarrow (\D, \D \setminus \Ss')$, 
and $k : \mcl{V} \hookrightarrow \D$ be inclusions. Now, by \eqref{E:D.metric.F.normal}, $\D$ is a normal space. Hence, the closed sets $\D \setminus \mcl{V}$ and $\Ss'$ have disjoint (open) neighborhoods. (In this book, ``disjoint'' means ``pairwise disjoint''.) Thus, $\overline{\D \setminus \mcl{V}} \subset \D \setminus \Ss'$, where $\overline{\D \setminus \mcl{V}}$ is the closure of $\D \setminus \mcl{V}$ and $\D \setminus \Ss'$ is open. Hence, we may apply excision (Munkres \cite[Theorem 31.7, p.\ 180]{jrM84}) (with $X = \D$, $A = \D \setminus \Ss'$, and $U = \D \setminus \mcl{V}$) to conclude that $j$ induces isomorphisms of homology. The following commutes with exact rows.
   \begin{equation*}
      \begin{CD}
         H_{r+1}(\D, \D \setminus \Ss') @>{\partial_{\ast}}>> H_{r}(\D \setminus \Ss') 
           @>{0}>> H_{r}(\D) \\
         @A{j_{\ast}}A{\isomto}A  @A{i_{\ast}}AA  @AA{k_{\ast}}A   \\
         H_{r+1}(\mcl{V}, \mcl{V} \setminus \Ss') @>>{\partial_{\ast}}> 
           H_{r}(\mcl{V} \setminus \Ss') @>>> H_{r}(\mcl{V}).
      \end{CD}
   \end{equation*}
(Here, ``0'' denotes the trivial map; see \eqref{E:inclusion.induces.triviality}; ``$\isomto$'' indicates isomorphism.)
A simple diagram chase shows that $i_{\ast}$ is surjective (in dimension $r$). Consequently, by theorem \ref{E:Phi.ast.image.includes.Theta.ast.image},
    \begin{equation*} 
      \bigl( \Phi \restriction_{\mcl{V} \setminus \Ss'} \bigr)_{\ast}
        \bigl[ H_{r}(\mcl{V} \setminus \Ss') \bigr] 
          = (\Phi_{\ast} \circ i_{\ast})\bigl[ H_{r}(\mcl{V} \setminus \Ss') \bigr]  
            = \Phi_{\ast}\bigl[ H_{r}(\D \setminus \Ss') \bigr]  \supset \Theta_{\ast} 
              \bigl[ H_{r}(\T) \bigr] \neq 0,
    \end{equation*}
by \eqref{E:nontriv.r-dim.homol}. Statement \ref{I:Phi.of.U.less.S.has.non0.r.homol} follows..
    \end{proof}
\begin{proof}[Proof of Theorem \ref{T:Phi.star.Hr.contains.Theta.star.Hr}]
First, assume $t > 0$. Let $R$ be the coefficient ring. Let $o \in H_{t}(\T) = H_{t}(\T;R)$ be an orientation class. That means the following (Dold \cite[p.\ 292; 4.1, p.\ 267]{aD95.alg.topol}). First, suppose $\T$ is not orientable, so $R$, has characteristic 2. Let $o \in H_{t}(\T)$ be the element corresponding (under the isomorphism $J_{\T} : H_{t}(\T;R) \to \Gamma(\T;R)$; 
Dold \cite[(2.7), p.\ 254 and Proposition 3.3, p.\ 260]{aD95.alg.topol}) to the canonical section that takes $x \in \T$ to the element of $H_{t} \bigl( \T, \T \setminus \{ x \}; R \bigr)$ that corresponds to $1 = -1 \in R$. Next, suppose 
$\T$ is orientable, let $ \mathbb{Z}$ be the integers, and let 
$O \in \Gamma(\T) = \Gamma(\T; \mathbb{Z})$ map $\T$ into $\tilde{\T}$ be an orientation (Dold \cite[Definition 2.9, p.\ 254]{aD95.alg.topol}), let $O_{R} = O \otimes 1 \in \Gamma(\T;R)$ map $\T$ into $\tilde{\T} \otimes R$. Let $o$ correspond to $O_{R}$ under $J_{\T}^{-1}$ (Dold \cite[(2.7), p.\ 254, Proposition 3.3, p.\ 260]{aD95.alg.topol}). Either way, let ``$\smallfrown o$'' denote the homomorphism given by cap product with $o$ (Dold \cite[Chapter VIII, chapter 7, pp.\ 291--292]{aD95.alg.topol}). 

By \textbf{hypothesis \ref{Hyp:S.cap.T.small}} of the theorem 
and \eqref{E:cohom.and.codim} 
$\check{H}^{t-r}(\T \cap \Ss') = 0$, where ``$\check{\;\,}$'' indicates \v{C}ech cohomology 
(Dold \cite[Chapter VIII, chapter 6]{aD95.alg.topol}). 
Let $i : (\T, \varnothing) \hookrightarrow (\T, T \cap \Ss')$, $j : \T \setminus \Ss' \hookrightarrow \T$,
$k : \T \setminus \Ss' \hookrightarrow \D \setminus \Ss'$, 
and $\ell : \T \cap \Ss' \hookrightarrow \T$ be inclusions. 
By Dold \cite[(6.10), p.\ 284]{aD95.alg.topol} the following diagram 
is exact at $\check{H}^{t-r}(\T)$. (By Dold \cite[Proposition 1.3, p.\ 248]{aD95.alg.topol} $\T$ is an ENR.)  The indicated isomorphisms (``$\isomto$'') are due to Poincar\'{e}-Lefschetz duality  
(Dold \cite[Proposition 7.2, p.\ 292]{aD95.alg.topol}; since, by \textbf{hypothesis \ref{Hyp:S.cap.T.closed}}, $\Ss' \cap \T$ is closed and $\T$ is compact, $\Ss' \cap \T$ is compact; if $\T \cap \Ss' = \varnothing$, use Poincar\'{e} duality, Dold \cite[8.1, p.\ 298]{aD95.alg.topol}). (By \eqref{E:(co)hom.of.empty.set.is.0}, the top row of the diagram is exact even if $\T \cap \Ss' = \varnothing$.) 

As for commutativity of the diagram, the only doubtful part is the square marked ``?''. 
The commutativity of that square follows from Dold \cite[(7.6), p.\ 293]{aD95.alg.topol}.\footnote{Here are the details. Dold \cite[(7.6), p.\ 293]{aD95.alg.topol} is actually a little ambiguous. In its first appearance in the formula, $i'$ is the inclusion map 
$(M \setminus L, M \setminus K) \hookrightarrow (M \setminus \tilde{L}, M \setminus \tilde{K})$. 
In its second appearance, $i'$ is the inclusion $(M, M \setminus K) \hookrightarrow (M, M \setminus \tilde{K})$. Apply Dold \cite[(7.6), p.\ 293]{aD95.alg.topol} with $(\T, \T, \T \cap \Ss', \T, \varnothing)$ in place of $(M, K, L, \tilde{K}, \tilde{L})$, respectively. The ``$i$'' in Dold \cite[(7.6), p.\ 293]{aD95.alg.topol} is 
$i : (\T, \varnothing) \hookrightarrow (\T, T \cap \Ss')$. The first ``$i'$'' in Dold \cite[(7.6), p.\ 293]{aD95.alg.topol} is $j : \T \setminus \Ss' \hookrightarrow \T$, and the second ``$i'$'' in Dold \cite[(7.6), p.\ 293]{aD95.alg.topol} is just the identity $: \T \to \T$.} 

\small
\begin{multline} \label{E:big.CD2}
  \begin{CD}
     0 \\                 
     @| \\                
     \check{H}^{t-r}(\T \cap \Ss') @<{\check{\ell}}<< \check{H}^{t-r}(\T) 
        @<{\check{i}}<< \check{H}^{t-r}(\T, \T \cap \Ss') @>{\isomto}>{\smallfrown o}> 
            H_{r}(\T \setminus \Ss')  @=  \\
     & & @V{\isomto}V{\smallfrown o}V ?  & &  @VV{j_{\ast}}V &  \\       
     & & H_{r}(\T) @= H_{r}(\T) @= H_{r}(\T) @>>{\Theta_{\ast}}> 
  \end{CD} \\
  %
  %
  \\
  \\
    \begin{CD}
      @= H_{r}(\T \setminus \Ss')   @>{k_{\ast}}>>   H_{r}(\D \setminus \Ss')  \\   
     & & @VV{\left( \Phi \restriction_{\T \setminus \Ss'} \right)_{\ast}}V @VV\Phi_{\ast}V \\
      @>>{\Theta_{\ast}}>           H_{r}(\F) @= H_{r}(\F)             
  \end{CD} \quad.
\end{multline}
\normalsize

The proof is a diagram chase. In \eqref{E:big.CD2} let $y = \Theta_{\ast}(x) \in H_{r}(\F)$, where $x \in H_{r}(T)$. Let $z \in \check{H}^{t-r}(\T)$ be the inverse under 
$\smallfrown o$ of $x$. Since $\check{H}^{t-r}(\T \cap \Ss') = 0$, by exactness $z = \check{i}(w)$ for some 
$w \in \check{H}^{t-r}(\T, \T \cap \Ss')$. Since the square marked ``?'' commutes, $j_{\ast}(w \smallfrown o) = x$. 
Therefore,
    \begin{equation} \label{E:Phi.star.circ.k.w.cap.o.=.y}
      \Phi_{\ast} \bigl( k_{\ast}(w \smallfrown o) \bigr) 
        = \left( \Phi \restriction_{\T \setminus \Ss'} \right) _{\ast} (w \smallfrown o) 
         = \Theta_{\ast} \circ j_{\ast} (w \smallfrown o)  
           = \Theta_{\ast}(x) = y.
    \end{equation}
Note that $k_{\ast}(w \smallfrown o) \in H_{r}( \D \setminus \Ss' )$.

Now take $t = 0$. Then, by hypotheses \ref{Hyp:T.manif} and \ref{Hyp:r.integer} of the theorem, we have $r = 0$ and $\T$ is a finite collection of points. By \textbf{hypothesis \ref{Hyp:S.cap.T.small}}, 
$\Ss' \cap \T = \varnothing$. It is obvious that $\check{H}^{0}(\T) \isomto H^{0}(\T)$ since $\T$ is an open cover of itself that refines every open cover. By universal coefficients, 
Munkres \cite[Theorem 53.1, p.\ 320]{jrM84}, $H^{0}(\T) \isomto H_{0}(\T)$. 
Moreover, the square ``?'' now trivially commutes. Hence, the argument goes through as before. 
\end{proof}

  \begin{remark}  \label{R:retractions}
Suppose $\Ss' = \varnothing$, so $\D' = \D$. Suppose $\Theta : \T \to \F$ is actually a homeomorphism. Then $\Phi$ is essentially the same thing as the retraction 
$\Theta^{-1} \circ \Phi : \D \to \T$. Then \eqref{E:Phi.ast.image.includes.Theta.ast.image} follows for all $r \geq 0$. (Such a $\Phi$ can be defined as a ``minimum distance method'' w.r.t.\ some metric. See remark \ref{R:dist.based.mthds}.) This leads to a contradiction if
$H_{r}(\D) = \{  0 \} \ne H_{r}(\T)$ for some $r$.
  \end{remark}
   
  \begin{remark}[Null sets]   \label{R:null.sets}
If these results only depend on a subset of $\D$ of probability 0 (``null set'') then they are without statistical interest. While no nonempty set is a null set for every probability measure, if one could escape these results by changing $\Phi$ on a nowhere dense set then our results would be vitiated. We show that, in fact, one cannot escape our results by changing $\Phi$ on any nowhere dense set. 
Suppose $\Phi : \D \setminus \Ss' \to \F$ is a data map satisfying the hypotheses 
of theorem \ref{T:Phi.star.Hr.contains.Theta.star.Hr}. Formally, let $\X \subset \D$ have empty interior and suppose by changing $\Phi$ 
only on $\X'$ one could then extend $\Phi$ continuously to all of $\D$. Call the continuous extension $\Psi$. We show that $\Psi = \Phi$ off $\Ss'$. 
Let $x_{0} \in \D \setminus \Ss'$. Since $\X$ empty interior and $\Ss'$ is closed, every neighborhood of $x_{0}$ contains a point not in $\Ss'$ or $\X$. It follows from 
 the continuity of $\Phi$ off $\Ss'$ and the continuity of $\Psi$ that 
 $\Psi(x_{0}) = \Phi(x_{0})$. Thus, $\Psi = \Phi$ off $\Ss'$. Hence, $\Psi$ also satisfies the hypotheses of theorem \ref{T:Phi.star.Hr.contains.Theta.star.Hr} and all the results that flow out of that theorem, like corollary \ref{C:consequences.of.H.r.D.=.0}.
  \end{remark}

\begin{remark}[Regularization]  \label{R:regularization.generalities}
Regularization (Tikhonov and Arsenin \cite{anTvyA77.illposed}) works for inverse problems and depends on being able to measure how well a proposed solution fits the data. (We use the term ``regularization'' broadly to also include the method of ``quasisolutions,'' Tikhonov and Arsenin \cite[p.\ 33]{anTvyA77.illposed}. See also  Mukherjee \emph{et al} \cite[Section 6.7]{sMrRtP03.Rglrzation}. Regularization is also connected to ``complexity'' of solutions, remark \ref{R:complexity}) In our general setup (subsection \ref{SS:calibration}) the idea of a ``good fit'' of solution to data only enters in the extremely weak form of calibration. 

(Regularization is also used in statistical learning to reduce the ``complexity'' of a statistical model. If a model is too complex then fitting it from data can lead to ``overfitting''. But overfitting is a form of excessive sensitivity of the model fitting data map to the data and so is similar to what we consider here. See Mukherjee \emph{et al} \cite[Section 6.7]{sMrRtP03.Rglrzation} again.)

Regularization will also take a very general form here. Suppose for some $b \in [0,\infty]$ we have a function $\Psi : \D \times [0, b) \partlyto \F$ and for 
$\lambda \in [0,b)$ let $\Phi_{\lambda} : x \mapsto \Psi(x,\lambda)$ 
($x \in \D$) whenever defined. Suppose $\Phi_{\lambda}$ is defined and continuous 
everywhere in $\D$ for $\lambda \in [0,b)$ sufficiently large. (Assume the following monotonicity property: $0 \leq \lambda < \lambda' < b$ with $\Phi_{\lambda}$ defined and continuous on $\D$ implies the same is true for $\Phi_{\lambda'}$.) 
In general, let $\Ss'_{\lambda} \subset \D$ be s.t.\  
$\lambda \in [0,b)$ $\Phi_{\lambda}$ is defined and continuous off $\Ss'_{\lambda}$. For $\lambda$ sufficiently large, then, we may take $\Ss'_{\lambda} = \varnothing$. In this case say that $\Phi_{\lambda}$ is ``regularized''. 
(In this remark, $\lambda \geq 0$ is an independent parameter, i.e., $\lambda$ is not computed from the data $x \in \D$.)

It may be hard to compute $\Psi$. Regularization often involves minimization (lemma \ref{L:data.maps.defined.by.opt}) and $\Psi$ may have singularities if the minimization problem is not strictly convex. But if one can compute $\Psi$, then one uses the data map 
$\Phi_{\lambda}$ with $\lambda$ large enough that $\Phi_{\lambda}$ is continuous.

But using $\Phi_{\lambda}$ with large $\lambda$ can lead to problems. 
Let $\T \subset \D$ and let $\Sigma : \T \to \F$ be a continuous ``standard'' for the problem at hand (section \ref{SS:calibration}). Given $\lambda \in [0, b)$, we would like 
$\Phi_{\lambda}$ to approximate $\Sigma$ on $\T \setminus \Ss'_{\lambda}$. Suppose, for example, $H_{r}(\D) = 0$ but 
    \begin{equation} \label{E:Sigma.non0.r.homol}
      \Sigma_{\ast} \bigl[ H_{r}(\T) \bigr] \neq 0 .
    \end{equation}
(In remarks \ref{R:regularization.on.spheres} and \ref{R:aug.mean.regularization}, we discuss regularization in a context in which $H_{r}(\D) \neq 0$.)

Adding the subscript $\lambda$ wherever appropriate, suppose the hypotheses of theorem \ref{T:Phi.star.Hr.contains.Theta.star.Hr} hold for every 
$\lambda \in [0,b)$. Let $\Theta_{\lambda}$ be the unique continuous extension 
of $\Phi_{\lambda} \restriction_{\T \setminus \Ss'_{\lambda}}$ to $\T$. 
Suppose $\Theta_{0}$ approximates $\Sigma$ well enough that 
$\Theta_{0\ast} \bigl[ H_{r}(\T) \bigr] = \Sigma_{\ast} \bigl[ H_{r}(\T) \bigr] \neq 0$. Then by corollary \ref{C:consequences.of.H.r.D.=.0}\eqref{I:S.still.aint.empty},  
$\Ss_{0} \neq \varnothing$.

Choose $\lambda_{1} \in (0,b)$ s.t.\ if $\lambda > \lambda_{1}$ then 
$\Phi_{\lambda}$ is  defined and continuous everywhere on $\D$. So   
$\Phi_{\lambda}$ has no singularities if $\lambda > \lambda_{1}$, i.e., 
$\D \setminus \Ss_{\lambda} = \D$. Therefore, 
if $\lambda \in (\lambda_{1}, b)$, we must have
$\Theta_{\lambda \ast} \bigl[ H_{r}(\T) \bigr] = 0 
\neq \Sigma_{\ast} \bigl[ H_{r}(\T) \bigr] = \Theta_{0 \ast} \bigl[ H_{r}(\T) \bigr]$.
In particular, \eqref{E:nontriv.r-dim.homol} fails for $\Theta = \Theta_{\lambda}$. 
Thus, if $\Phi_{\lambda}$ is continuous, then it must be uncalibrated, i.e., it does not approximate $\Sigma$ well on $\T$. In particular, the function 
$(x, \lambda) \mapsto \Theta_{\lambda}(x)$ cannot be continuous 
in $(x, \lambda) \in \D \times [0, b)$. For otherwise, $\Theta_{\lambda}$ and $\Theta_{0}$ would be homotopic and 
\eqref{E:nontriv.r-dim.homol} would hold for $\Theta = \Theta_{\lambda}$ with $\lambda > \lambda_{1}$. Thus, as $\lambda \downarrow 0$, at some critical value $\lambda = \lambda_{0} \geq 0$ the data map $\Phi_{\lambda}$, must undergo a ``bifurcation'' (Strogatz \cite[Chapter 3]{shS1998.strogatz}). 

$\Phi_{\lambda}$ cannot be calibrated for $\lambda$ greater than the critical point 
$\lambda_{0}$. But in order for $\Phi_{\lambda}$ to be considered a solution of the data analysis problem at hand, $\Theta_{\lambda}$ must approximate $\Sigma$. One can
hope that as $\lambda \downarrow \lambda_{0}$, the map $\Phi_{\lambda}$ will at least be approximately calibrated. But if \eqref{E:nontriv.r-dim.homol} fails then the approximation of $\Theta_{\lambda}$ to $\Sigma$ cannot be very good. 
The most we can require is that for $\lambda$ only slightly bigger than 
$\lambda_{0}$ the map $\Theta_{\lambda}$ be close to $\Sigma$ on all but a small subset, $\T_{\lambda}$, of $\T$. \eqref{E:Sigma.non0.r.homol} says that $\Sigma$ wraps $\T$ around a void in $\F$. If $\lambda > \lambda_{0}$ then $\Theta_{\lambda}$ cannot do this. However, since $\Theta_{\lambda}$ is 
close to $\Sigma$ off $\T_{\lambda}$, 
$\Theta_{\lambda}$ must \emph{almost} wrap $\T \setminus \T_{\lambda}$ around the void in $\F$. But that means in the small set $\T_{\lambda}$ the map 
$\Theta_{\lambda}$ must almost completely \emph{unwrap} $\T$ from around the void in $\F$. In that region $\Theta_{\lambda}$ will not have singularities, but it will still be unstable in a quantitative sense because $\T_{\lambda}$ is small. 

Moreover, suppose $\lambda > \lambda_{0} > \lambda' \geq 0$, 
$\lambda - \lambda' > 0$ is small, and $x \in \Ss_{\lambda'}$. Then $x$ will not be a singularity of $\Phi_{\lambda}$, but $\Phi_{\lambda}$ may still be unstable near $x$ in a quantitative sense. Borrowing a term from dynamical systems (Strogatz \cite[p.\ 99]{shS1998.strogatz}), one might call this lingering instability the ``ghost'' of the singularity of $\Phi_{\lambda'}$ at $x$. Thus, $\Phi_{\lambda}$ can have unstable regions distant from $\T_{\lambda}$, too. 

A number of things that can go wrong in regularization:  A non-convex minimization problem can have unstable solutions, 
$\Phi_{\lambda}$ can be poorly calibrated, $\Phi_{\lambda}$ can be unstable on part of $\T$, or ghosts of singularities can haunt $\Phi_{\lambda}$. 
The theory described in  this chapter explains the connection among these phenomena. 

See remark \ref{R:regularization.on.spheres} for discussion of regularization in a specific context. A specific case is discussed in remark \ref{R:aug.mean.regularization}.
  \end{remark}

\section{Dimension of singular sets}  \label{SS:dim.of.sing.sets}          
(My memory of this is hazy, but I believe that I got the idea of looking at the dimension of singular sets from D. Ravenel; personal communication.) So far we have just discussed the existence of singularities. However, sometimes one can take \eqref{E:Phi.ast.image.includes.Theta.ast.image} 
and \eqref{E:nontriv.r-dim.homol} and go on to compute lower bounds on the dimension of $\Ss'$ or even the Hausdorff measure of $\Ss'$ (appendix \ref{Chptr:Lip.Haus.meas.dim}, chapter \ref{Chptr:Haus.meas.of.sing.set}). 

Continue using the notation from theorem \ref{T:Phi.star.Hr.contains.Theta.star.Hr}. Recall that ``$\dim$'' denotes Hausdorff dimension and the ``codimension'' of $\Ss'$, relative 
to $\D$, is $\text{codim} \, \Ss' := \dim \D - \dim \Ss'$. 
Suppose $d := \dim \D > t := \dim \T$.
  
Suppose one can prove
	\begin{equation}  \label{E:check.H.d-r-1.not.0}
		\check{H}^{d-r-1}(\Ss'; R) \ne 0,
	\end{equation}
where $R$ is the coefficient ring. Then, by \eqref{E:cohom.and.codim},
$\Hm^{d-r-1}(\Ss') > 0$. In particular, $\text{codim} \, \Ss' \leq r + 1$. 

Suppose we have $H_{r}(\D) = 0$. (So $r > 0$.) By Poincar\'{e} duality (Dold \cite[Proposition VIII.8.1, p.\ 298]{aD95.alg.topol},
if $\D$ is a compact manifold then $H_{r}(\D)$ is trivial implies 
$\check{H}^{d-r}(\D) = 0$. In that case an easy argument (the proof of proposition \ref{P:sing.dim.when.H.d-r.D.=.0} below) shows that \eqref{E:check.H.d-r-1.not.0} 
holds. 
It follows that it is a Euclidean neighborhood retract (END) by assumption it is an ENR (Dold \cite[Proposition and Definition IV.8.5, p.\ 81 and Proposition VIII.1.3, p.\ 248]{aD95.alg.topol}). Therefore, by Dold \cite[Proposition VIII.6.12, p.\ 285]{aD95.alg.topol}, $\check{H}^{t-r}(\T) \isomto H^{t-r}(\T)$, singular cohomology. For future reference:
    \begin{equation}  \label{E:sing.and.cech.cohoms.same}
      \text{Singular and \v{C}ech cohomologies of a compact manifold are the same.}
    \end{equation}

In the proof of theorem \ref{T:spher.loc.sing.codim.nvar+1} we arrive at \eqref{E:check.H.d-r-1.not.0} (essentially) in a case in which $H_{r}(\D) \neq 0$.
   \begin{prop}  \label{P:sing.dim.when.H.d-r.D.=.0}
Suppose the hypotheses of theorem \ref{T:Phi.star.Hr.contains.Theta.star.Hr} and \eqref{E:nontriv.r-dim.homol} hold with $0 < t < d$ and $r = 1, \ldots, t-1,$ or $t$. Suppose 
$\D$ is a compact $d$-dimensional manifold with 
$\check{H}^{d-r}(\D) \isomto H^{d-r}(\D) = \{ 0 \}$ and $\Ss'$ is closed in $\D$. 
Then $\check{H}^{d-r-1}(\Ss')$ is nontrivial and, hence, 
         \begin{equation}  \label{E:codim.S.leq.r+1}
            \Hm^{d-r-1}(\Ss') > 0. \; \text{ In particular, } 
              \text{codim} \, \Ss' \leq r + 1 .
         \end{equation}
   \end{prop}
As in theorem \ref{T:Phi.star.Hr.contains.Theta.star.Hr}, if $\T$ is not orientable, use a commutative ring of characteristic 2 for coefficients. Otherwise, any commutative coefficient ring is permissible. For the $r = 0$ case see example \ref{Ex:disconnected.F}.

By Poincar\'{e} duality (Dold \cite[Proposition VIII.8.1, p.\ 298]{aD95.alg.topol}, 
$H^{d-r}(\D) \isomto H_{r}(\D)$. In particular, the condition $H^{d-r}(\D) = \{ 0 \}$ cannot hold for $r = 0$. 

Note that in contexts where the notion of ``submanifold'' of $\D$ makes sense, $\T$ does not have to be one. By Poincar\'{e} duality again,  
if $\D$ is a compact manifold and $\check{H}^{d-r}(\D) = 0$ then $H_{r}(\D) = 0$, an impossibility unless $r > 0$. By \textbf{hypothesis \ref{Hyp:r.integer}} of theorem \ref{T:Phi.star.Hr.contains.Theta.star.Hr}, $r \leq t$. This explains the clause 
``$0 < t < d$ and $r = 1, \ldots, t$'' in the proposition.
The case $r = 0$ is treated in example \ref{Ex:disconnected.F}.
  \begin{proof}
By theorem \ref{T:Phi.star.Hr.contains.Theta.star.Hr} and \eqref{E:nontriv.r-dim.homol} $H_{r}(\D \setminus \Ss')$ is nontrivial. Therefore, by Poincar\'{e}-Lefschetz) duality (Dold \cite[Proposition VIII.7.2, p.\ 292]{aD95.alg.topol}) with $K = M = \D$, $L = \Ss'$, $n = d$, and $i = d-r$, we have 
$\check{H}^{d-r}(\D, \Ss') \isomto H_{r}(\D \setminus \Ss')$ is nontrivial. 
But the following is exact (Dold \cite[Proposition VIII.6.10, p.\ 284]{aD95.alg.topol}).
    \begin{equation} \label{E:S'.to.D,S'.to.D.exact.seq}
      0 = \check{H}^{d-r}(\D) \leftarrow \check{H}^{d-r}(\D, \Ss') 
        \leftarrow \check{H}^{d-r-1}(\Ss') .
    \end{equation}
Hence, if $\check{H}^{d-r-1}(\Ss')$ were trivial $\check{H}^{d-r}(\D)$ would be nontrivial. Contradiction. So $\check{H}^{d-r-1}(\Ss')$ is nontrivial. Now, $\Ss'$ is a closed subset of the compact manifold $\D$. Therefore $\Ss'$ is compact and we may apply \eqref{E:cohom.and.codim}. 
  \end{proof}
Notice that the proof really only requires that $\check{H}^{d-r}(\D)$ not contain a subgroup isomorphic to $\check{H}^{d-r}(\D, \Ss')$. This idea is used in example \ref{Ex:disconnected.F}. 

  \begin{remark}
Suppose \eqref{E:codim.S.leq.r+1} holds.
Since $d > t$ are integers, we have $d-r-1 \geq t-r$. Therefore, 
   \begin{equation} \label{E:general.fallback}
      \Hm^{t-r}(\Ss') > 0, \text{ so } \dim \Ss' \geq t - r. 
   \end{equation}
If \textbf{hypothesis \ref{Hyp:S.cap.T.small}} of theorem \ref{T:Phi.star.Hr.contains.Theta.star.Hr} fails then \eqref{E:general.fallback} still holds. So if one wants to bound $\dim \Ss'$ below even if \textbf{hypothesis \ref{Hyp:S.cap.T.small}} of theorem \ref{T:Phi.star.Hr.contains.Theta.star.Hr} fails then one should try to find a high dimensional test pattern space (i.e., large $t$) for which all but \textbf{hypothesis \ref{Hyp:S.cap.T.small}} of theorem \ref{T:Phi.star.Hr.contains.Theta.star.Hr} all hold. Then whether or not \textbf{hypothesis \ref{Hyp:S.cap.T.small}} holds one still has 
$\dim \Ss' \geq t - r$. In section \ref{SS:gnrl.lwr.bnd.plane.fit} this idea is carried through for ``plane fitting''.

Here is a case in which the two lower bounds, \eqref{E:codim.S.leq.r+1} and \eqref{E:general.fallback}, coincide. If $\D$ is a smooth manifold and $\T$ is a compact imbedded submanifold of $\D$, let $\hat{\T}$ be the boundary of a tubular neighborhood of $\T$ in $\D$ (section \ref{SSS:tubular.nbhd} below). So $\dim \hat{\T} = d-1$. Then with $\hat{\T}$ in place of $\T$, the bound on $\dim \Ss'$ in \eqref{E:general.fallback} is just $d - r - 1$, the same as the lower bound in \eqref{E:codim.S.leq.r+1}.
  \end{remark}

  \begin{remark}[Manifolds with boundary]  \label{R.doubling.manifs}
Proposition \ref{P:sing.dim.when.H.d-r.D.=.0} requires $\D$ to be compact manifold, which in particular have no boundary. Theorem \ref{T:lwr.bnd.on.Haus.meas} below imposes the same requirement. However, it is not hard to extend them to compact manifolds with \emph{boundary}. One just uses ``doubling'' (Munkres \cite[Definition 5.10, pp.\ 56--57]{jrM66}) to create a manifold without boundary. For example suppose $\tilde{\D} = \{ x \in \RR^{d}: |x| \leq 1 \}$ is a $d$-ball with 
$d \geq 1$ and $\T$, a $(d-1)$-sphere, is the boundary of $\tilde{\D}$. Suppose $\F$ is a $(d-1)$-sphere and $\tilde{\Phi} : \tilde{\D} \setminus \Ss' \to \F$ is continuous, where $\Ss' \subset \tilde{D}$ is closed. (Such functions have received much attention, 
Brezis \cite{hB03.AnlysTopol}.)   Suppose the restriction 
$\tilde{\Phi} \restriction_{\T}$ (see \eqref{E:restriction.notation}) has a continuous extension, $\Theta$, to $\T$ and $\Theta$ has nonzero degree 
(Munkres \cite[p.\ 116]{jrM84}). 
Paste two copies of $\tilde{\D}$ together along  $\T$ to create a new $d$-sphere, 
$\D$. $\tilde{\Phi}$ (partially) extends to $\D$ in the obvious way. Call the extension $\Phi$. One can then try to apply our results to 
$\Phi$. The conclusions have obvious interpretation for $\tilde{\Phi}$. 
  \end{remark}

  \begin{remark}[Closedness of singular set] \label{R:Ss.closed?}
Note that $\Ss'$ is a \emph{superset} of the singular set, $\Ss$, of $\Phi$. Some care is needed to insure that in applying the results of this chapter $\Ss'$ is an interesting set. If $\Ss$ is closed then we may take $\Ss' = \Ss$. But $\Ss$ need not be closed as the LAD example in figure \ref{F:LS.PC.LAD.lf.plots} shows. Another example is as follows. Suppose $\D = \RR^{2}$ and $\F = \RR$. Let 
	\[
		\Phi(x,y) =
			\begin{cases}
				x, &\text{ if } y \geq 0 \\
				0, &\text{ if } y < 0.
			\end{cases}
	\]
The singular set of this $\Phi$ is $\RR$ with the point 0 removed -- not a closed set. (See also remark \ref{R:S.measurable}.)

Of course, we can take $\Ss'$ to be the closure of $\Ss$. However, as example \ref{E:different.dimensions} shows, the closure of a set can be very different, even in dimension, than the original set. So one might try to check that $\Ss$ is closed. But requires studying the behavior of $\Phi$ far, possibly, from $\T$ or even the perfect fit set $\Pf$. That makes a mockery of the philosophy underlying this project laid out in remark \ref{R:sales.pitch}, viz. to only have to examine $\Ss$ near a small set 
$\T$ or $\Pf$. (See section \ref{SS:singularity}.) The process of proving $\Ss$ is closed may  reveal $\dim \Ss$, rendering the proposition unnecessary.

In practice, the closure, $\overline{\Ss}$ might be essentially the same as $\Ss$, differing by some sort of null set. However, that just translates the problem to that of demonstrating that $\overline{\Ss}$ might be essentially the same as $\Ss$, which again involves $\Ss$ not just in the vicinity of $\T$ or $\Pf$. 

However, in chapter \ref{Chptr:severity} we show that, under hypotheses on $\F$ that seem to often hold in practice, 
we may take $\Ss'$ to be a closed \emph{subset} of $\Ss$ consisting only of ``severe'' singularities. More precisely, in theorem \ref{T:if.lin.combo.on.F.then.can.rstrct.to.bad.sings}  
describes an operation that takes the pair, $(\Phi, \D')$, to a pair, $(\Omega, \tilde{\D})$ ($\tilde{\D} \subset \D$) and $\Omega : \tilde{D} \to \F$ 
s.t.\ $\D' \subset \tilde{\D}$,  
$\Omega \restriction_{\T \cap \D'}= \Phi \restriction_{\T \cap \D'}$ (see \eqref{E:restriction.notation}), the singular set of $\Omega$ lies in a closed subset, 
$\Ss^{\msf{V}}$, of $\Ss$ consisting of ``severe'' singularities of $\Phi$, 
and $\tilde{D} = \D \setminus \Ss^{\msf{V}}$. Now apply theorem \ref{T:Phi.star.Hr.contains.Theta.star.Hr} with $\Omega$ in place of $\Phi$. 
\end{remark}
   
  \begin{example}[Disconnected $\F$]  \label{Ex:disconnected.F} 
Apparently, the $r = 0$ version of proposition \ref{P:sing.dim.when.H.d-r.D.=.0} needs to be treated separately. So take 
    \begin{equation*}
      r = 0 .
    \end{equation*}
Suppose $\F$ has mutiple path-connected components. 
Let $\Ss' \subset \D$ be closed with empty interior. 
Suppose $\Phi : \D \setminus \Ss' \to \F$ is continuous but $\Phi(\D \setminus \Ss)$ is not a subset of just one path component of $\F$. Hypothesis testing (example \ref{Ex:hypothesis.testing}) is an example. 
Let $m = 2, 3, \ldots$. 
Pick $x_{1}, \ldots, x_{m} \in \D \setminus \Ss'$ s.t.\ $\Phi(x_{i})$ ($i=1, \ldots, m$) belong to distinct path-connected components of $\F$. Assume \eqref{E:D.metric.F.normal} holds as usual, so $\D$ is pathwise connected. Then obviously, $\Phi : \D \setminus \Ss' \to \F$ must have singularities. 

We show how this situation fits in with our results so far. 
Let $\T := \{ x_{1}, \ldots, x_{m} \}$. This choice of $\T$ is somewhat arbitrary. The data sets $x_{1}, \ldots, x_{m}$ do not really qualify as ``perfect''. (This point also comes up in section \ref{SS:calibration} and chapter \ref{Chptr:linear.classification}.) The \textbf{hypotheses} of theorem \ref{T:Phi.star.Hr.contains.Theta.star.Hr} hold. 

Since $\Ss' \cap \T = \varnothing$, we may let 
$k : \T \hookrightarrow \D \setminus \Ss'$ be inclusion and 
$\Theta := \Phi \circ k : \T \to \F$.  
By Munkres \cite[Theorem 29.4, p.\ 164]{jrM84}, $\tilde{H}_{0}(\T; \ZZ)$ is a nontrivial free group and
$\Theta_{\ast} : \tilde{H}_{0}(\T; \ZZ) \to \tilde{H}_{0}(\F; \ZZ)$ is injective. 
Thus, \eqref{E:nontriv.r-dim.homol} holds with $r = 0$. The same theorem tells us that, since $\D$ is path-connected, $H_{0}(\D; \ZZ) \isomto \ZZ$. 
By remark \ref{R:nontriv.Theta.star.in.dim.0} 
(specifically \eqref{E:H0(T).has.rank.>.1}), 
$H_{0}(\D \setminus \Ss'; \ZZ)$ is free with rank $> 1$. Thus, by Munkres \cite[Theorem 29.4, p.\ 164]{jrM84} again, $\D  \setminus \Ss$ is not pathwise connected and $\Ss \neq \varnothing$. 

Now suppose $\D$ is a compact connected $d$-dimensional manifold with $d > 0$. We make no extra assumptions about $\D$, e.g.\ concerning its (co)homology. We already know $\Ss \neq \varnothing$ and $r = 0$ by assumption. 
It follows that, if $d = 1$ then $\Hm^{d-r-1}(\Ss') > 0$ and 
$\text{codim} \, \Ss' \leq 1 = r+1$. I.e., \eqref{E:codim.S.leq.r+1} holds.

Suppose $d > 1$. We have just observed that 
$H_{0}(\D; \ZZ) \isomto \ZZ$ and $H_{0}(\D \setminus \Ss'; \ZZ)$ is free with rank 
$> 1$. It follows by universal coefficients (Munkres \cite[Theorem 55.2, p.\ 332]{jrM84} or Spanier \cite[Theorem 8, p.\ 222]{ehS66}), that 
$H_{0}(\D; \ZZ/2)$ is a vector space of dimension 1 and 
$H_{0}(\D  \setminus \Ss'; \ZZ/2)$ is a vector space of dimension $> 1$. 

Let $R := \ZZ$ if $\D$ is oriented. Otherwise, let $R := \ZZ/2$. (Thus, whether 
$R = \ZZ/2$ or $\ZZ$, $H_{0}(\D \setminus \Ss'; R)$ and $H_{0}(\D; R)$ are free modules over a principal ideal domain.) We argue as in the proof of proposition \ref{P:sing.dim.when.H.d-r.D.=.0}. By Dold \cite[Proposition VIII.6.10, p.\ 284]{aD95.alg.topol}, the following is exact. 
    \begin{equation} \label{E:r=0.S'.to.D,S'.to.D.exact.seq}
      \check{H}^{d}(\D; R) \leftarrow \check{H}^{d}(\D, \Ss'; R) 
        \leftarrow \check{H}^{d-1}(\Ss'; R) .
    \end{equation}
Recall \eqref{E:sing.and.cech.cohoms.same}.    
Then by Dold \cite[Poincar\'{e} duality 8.1, p.\ 298]{aD95.alg.topol} with $M = \D$, 
$G = R$,  and $i = n = d$, we have $\check{H}^{d}(\D; R) \isomto H_{0}(\D; R)$. Therefore, $\check{H}^{d}(\D; R)$ is free with rank/dimension 1. Arguing as in the proof of proposition \ref{P:sing.dim.when.H.d-r.D.=.0} we observe that $\check{H}^{d}(\D, \Ss'; R) \isomto H_{0}(\D \setminus \Ss'; R)$. Therefore, $\check{H}^{d}(\D, \Ss'; R)$ has rank/dimension $> 1$. 
 
It follows (by Munkres \cite[Theorem 4.2, p.\ 24]{jrM84} in the $R = \ZZ$ case; Stoll and Wong \cite[Theorem 2.1, p.\ 99]{rrSetW68.LinearAlgebra} in the $\ZZ/2$ case) from \eqref{E:r=0.S'.to.D,S'.to.D.exact.seq}, that $\check{H}^{d-1}(\Ss'; R)$ is nontrivial. Therefore, by \eqref{E:cohom.and.codim}, 
$\Hm^{d-r-1}(\Ss') = \Hm^{d-1}(\Ss') > 0$ so $\dim \Ss' \geq d-1$. Thus, \eqref{E:codim.S.leq.r+1} holds in the $r = 0$ case. 
    
    This small bound, $\text{codim} \, \Ss' \leq 1$, means that the singular set is large. That  is unfortunate because the case of disconnected $\F$ is very 
    important in practice. It encompasses choosing one of a discrete set of actions, conclusions, or decisions based on data. A very important special case is statistical hypothesis testing  (Lehmann \cite{elL93.StatHyps} and example \ref{Ex:hypothesis.testing} above), which for good or ill is the chief statistical activity in biomedical research.
      \end{example}
      
    $\check{H}^{d - r}(\D) = 0$ holds, e.g., if $\D$ is a sphere of dimension 
    $d > d - r \geq d-t > 0$. 
    In the plane fitting problem (chapter \ref{Chptr:sings.in.plane.fit}) proposition \ref{P:sing.dim.when.H.d-r.D.=.0} applies because we can restrict attention to spheres. Doing so also gives information about where the singularities lie: in the spheres. (See remark \ref{R:dnsity.contrs.as.data.spaces}.)
    
    Proposition \ref{P:sing.dim.when.H.d-r.D.=.0} does not seem to help in the directional location case (chapter \ref{Chptr:spherical.location}). For that example using theorem \ref{T:Phi.star.Hr.contains.Theta.star.Hr} to prove \eqref{E:check.H.d-r-1.not.0} requires an additional assumption, \eqref{E:assume.Phi.S'.sym}, and some work. 
    
      \begin{remark}[``Wiggling''] \label{R:wiggling}
    If $\D_{1}$ is a subspace of $\D$ that is a compact manifold in the relative topology then one might be able to compute a lower
    bound on $\dim (\D_{1} \cap \Ss')$, say, $\dim (\D_{1} \cap \Ss') \geq s_{1}$. If one can do this for an $s_{2}$-dimensional family,
    $\mathfrak{D}$, of  manifolds like $\D_{1}$, s.t.\ $\D_{1}', \D_{1}'' \in  \mathfrak{D}$ distinct implies $\dim \D_{1}' \cap \D_{1}'' < s_{1}$, 
    then one might be able to conclude that $\dim \Ss' \geq s_{1} + s_{2}$.
    
    Similarly, one can wiggle $\T$'s. Suppose one has an $s$-dimensional family, $\mathfrak{T}$, of compact, $t$-dimensional imbedded submanifolds of $\D$.
    Suppose that \textbf{hypothesis \ref{Hyp:S.cap.T.small}} of theorem \ref{T:Phi.star.Hr.contains.Theta.star.Hr} fails for every $\T \in \mathfrak{T}$. Then one needs to settle for \eqref{E:general.fallback}, i.e., $\dim (\T \cap \Ss') \geq t - r$ for every 
    $\T \in \mathfrak{T}$. Suppose for $\T, \T' \in \mathfrak{T}$ we have 
    $\dim (\T \cap \T') < t - r$. Then one may be able to conclude 
    that $\dim \Ss' \geq s + t - r$. An example of this can be found in section \ref{SS:gnrl.lwr.bnd.plane.fit}. A similar idea is used in chapter \ref{Chptr:linear.classification}.
    
    One might also play the same game for Hausdorff measure (Theorem \ref{T:lwr.bnd.on.Haus.meas}), not just Hausdorff dimension.
  \end{remark}

\section{More topology}  \label{SS:more.topol}
\subsection{Singularities in $\T$} Here we investigate further the homology groups of $\Ss'$ in the framework of proposition \ref{P:sing.dim.when.H.d-r.D.=.0}. Compute (co)homology using coefficients in a field, $F$, (of characteristic 2 if $\T$ is nonorientable). Thus, $\D$ is a compact manifold with $\check{H}^{d-r}(\D) = 0$. First, suppose \eqref{E:nontriv.r-dim.homol} and all hypotheses of 
of Theorem \ref{T:Phi.star.Hr.contains.Theta.star.Hr} hold, with the possible exception of  \ref{Hyp:S.cap.T.small}.  Suppose $\dim \Ss' < d-r-1$. Then by proposition \ref{P:sing.dim.when.H.d-r.D.=.0}, \textbf{hypothesis \ref{Hyp:S.cap.T.small}} of theorem \ref{T:Phi.star.Hr.contains.Theta.star.Hr} \emph{must} fail. In particular, 
$\dim (\Ss' \cap \T) \geq t-r$.
By \eqref{E:nontriv.r-dim.homol}, there exists $x \in H_{r}(\T)$ s.t.\ $\Theta_{\ast}(x) \in H_{r}(\F)$ is nonzero. Given $x$, under some regularity conditions, we find a  corresponding nontrivial linear map in $Hom_{F} \bigl( H_{t-r}(\Ss' \cap \T), F)$ corresponding to $x$. 

Since $\check{H}^{d-r}(\D) = 0$, $H_{r}(\D \setminus \Ss')$ must be trivial. For otherwise, by the proof of proposition \ref{P:sing.dim.when.H.d-r.D.=.0}, we have 
$\check{H}^{d-r-1}(\Ss') \ne 0$, contradicting the assumption $\dim \Ss' < d-r-1$. (See \eqref{E:cohom.and.codim}.) Therefore, by duality, $H_{r}(\D \setminus \Ss')$ is trivial. 

Refer to the commutative diagram \eqref{E:big.CD2}. Let $z \in \check{H}^{t-r}(\T)$ be the unique cohomology class s.t.\ $z \cap o = x$. If $\check{\ell}(z) \in \check{H}^{t-r}(\Ss' \cap \T)$ were trivial then the diagram chase in the proof of theorem \ref{T:Phi.star.Hr.contains.Theta.star.Hr} would show that $H_{r}(\D \setminus \Ss') \ne \{ 0 \}$, a contradiction. Therefore, $\check{\ell}(z)$ is a nontrivial class in $\check{H}^{t-r}(\Ss' \cap \T)$. 

Note that, since $\T$ is a compact manifold by assumption it is an ENR (Dold \cite[Proposition and Definition IV.8.5, p.\ 81 and Proposition VIII.1.3, p.\ 248]{aD95.alg.topol}). Therefore, by Dold \cite[Proposition VIII.6.12, p.\ 285]{aD95.alg.topol}, we have that 
$\check{H}^{t-r}(\T) \isomto H^{t-r}(\T)$, singular cohomology. In summary 
Suppose the space $\Ss' \cap \T$ is also sufficiently nice that 
$\check{H}^{t-r}(\Ss' \cap \T) \isomto H^{t-r}(\Ss' \cap \T)$. For example, this happens when the spaces are triangulable (Munkres \cite[Theorems 34.3, p.\ 194 and 73.2, p.\ 437]{jrM84}). 
E.g., suppose $\Ss' \cap \T$ is a semi-algebraic set (Bochnak \emph{et al} \cite[Theorem 9.2.1, p.\ 217]{jBmCr-rR98.RealAlgGeom}). Or suppose $\Ss' \cap \T$ is an ENR (see \eqref{E:sing.and.cech.cohoms.same}). if $\Ss' \cap \T$ is an ENR 
In any case, if $\check{H}^{t-r}(\Ss' \cap \T) \isomto H^{t-r}(\Ss' \cap \T)$ then it follows from the universal coefficients theorem for cohomology (Munkres \cite[Corollary 53.6, p.\ 326]{jrM84}) that there exists a nontrivial linear map $f \in Hom_{F} \bigl( H_{t-r}(\Ss' \cap \T), F)$ corresponding to $\ell^{\ast}(z)$. $f$ is the nontrivial linear map in $Hom_{F} \bigl( H_{t-r}(\Ss' \cap \T), F)$ corresponding to $x$. 

\subsection{Singularities bounded away from $\T$} Let $\clU \subset \D$ be an open neighborhood of $\T$. Suppose 
$\Ss' \cap \clU = \varnothing$. Moreover, assume that $\Ss'$, $\D$, and $\D \setminus \clU$ are sufficiently nice (e.g., triangulable or ENR)
that we may assume their singular and \v{C}ech cohomologies coincide. Here we show that under the hypotheses of proposition \ref{P:sing.dim.when.H.d-r.D.=.0} (so in particular $\Ss'$ is closed) inclusion $m : \Ss' \hookrightarrow \D \setminus \clU$ induces a nontrivial homomorphism in $(d-r-1)$-dimensional homology. First, note that by duality 
$H_{r}(\D)$ is trivial. Let $j: (\D, \clU) \hookrightarrow (\D, \D \setminus \Ss')$ 
and $k_{\clU} : \clU \hookrightarrow \D \setminus \Ss'$ be inclusions. Use coefficients in a field, $F$, (of characteristics 2 if $\D$ is nonorientable). Consider the following commutative diagram.  
\tiny
   \begin{equation}   \label{E:S.outside.U.comm.diag}
      \begin{CD}
          & & & & & & & & \{ 0 \}  \\
          & & & & & & & & @| \\
           \text{Hom}_{F} \bigl(H_{d-r-1}(\D \setminus \clU),F \bigr)  @<{\isomto}<< 
               H^{d-r-1}(\D \setminus \clU) 
                     @>{\isomto}>> H_{r+1}(\D, \clU) @>{\partial_{\ast}}>> H_{r}(\clU) 
                        @>>> H_{r}(\D) \\
                        @V{\text{Hom}(m_{\ast})}VV            @V{m^{\ast}}VV             
                            @V{j_{\ast}}VV             
                            @V{k_{\clU \ast}}VV        @| \\
           \text{Hom}_{F} \bigl(H_{d-r-1}(\Ss'),F \bigr) @<{\isomto}<< H^{d-r-1}(\Ss') 
             @>{\isomto}>> 
               H_{r+1}(\D, \D \setminus \Ss') 
                     @>{\partial_{\ast}}>> H_{r}(\D \setminus \Ss') @>>> H_{r}(\D)  \\
          & & & & & & & & @| \\
          & & & & & & & & \{ 0 \}
      \end{CD} 
   \end{equation}
\normalsize
The first rectangle on the left in \eqref{E:S.outside.U.comm.diag} comes from universal coefficients 
(Munkres \cite[Corollary 53.6, p.\ 326]{jrM84}). The second comes from duality 
(Dold \cite[2.8 p.\ 254 and pp.\ 292--293]{aD95.alg.topol}). The commutativity of the third rectangle from the left comes from the naturality of the exact homology sequence of a pair 
(Munkres \cite[Theorem 30.2, p.\ 169]{jrM84}). The top and bottom rows are exact 
at $H_{r}(\clU)$ and $H_{r}(\D \setminus \Ss')$, resp. by exact homology sequence of a pair again. Let 
$k : \T \hookrightarrow \D \setminus \Ss'$ be inclusion. By \eqref{E:nontriv.r-dim.homol} there exists $y \in \Theta_{\ast} \bigl[ H_{r}(\T) \bigr] \setminus \{ 0 \}$ Hence, by \eqref{E:Phi.star.circ.k.w.cap.o.=.y}, we have that $k_{\ast} : H_{r}(\T) \to H_{r}(\D \setminus \Ss')$ is nontrivial. But $k$ factors through $\clU$. Thus, 
$k_{\clU \ast} : H_{r}(\clU) \to H_{r}(\D \setminus \Ss')$ is nontrivial. A diagram chase around \eqref{E:S.outside.U.comm.diag} then shows that $\text{Hom}(m_{\ast})$ is nontrivial. Thus, $m_{\ast}$ is nontrivial, as desired.

\chapter{Measure}  \label{Chptr:Haus.meas.of.sing.set}
In this chapter we make extensive use of Hausdorff dimension and measure. These are defined and some properties of them given in appendix \ref{Chptr:Lip.Haus.meas.dim}.

We will see that in plane-fitting (chapter \ref{Chptr:sings.in.plane.fit}) important examples where the codimension of the singular set will never exceed 2, regardless of the number of data points, the number of variables involved, or the dimension of the plane we fit. This suggests that there is a lot of information about the singular set that is not captured by its dimension. Consider the situation described in proposition \ref{P:sing.dim.when.H.d-r.D.=.0}. In this chapter, we get more information about the singular set when the 
$\Hm^{d-r-1}$-essential distance (see \eqref{E:essential.dist.defn}) from $\Ss'$ to the space, 
$\Pf \subset \D$, of perfect fits is positive. 

In this chapter we will assume the following:  
	\begin{multline}  \label{E:D.is.cmpct.Riem.manif}
	      \D \text{ is a compact, connected $d$-dimensional Riemannian 
	        manifold with } d > 0 \\
	          \text{ and Riemannian metric } \langle \cdot, \cdot \rangle. 
	\end{multline} 
Define $\xi$ by:
	\begin{equation}  \label{E:xi.is.metric.on.D}
           \text{Let } \xi \text{ be the topological metric on } \D 
	      \text{ determined by the Riemannian metric on } \D.
	\end{equation}
By Boothby \cite[Corollary (7.11), p.\ 346]{wmB75},  
	\begin{equation}  \label{E:join.points.by.geod}
	      \text{Any pair } x, y \in \D \text{ can be joined by a geodesic 
	        whose length is } \xi(x,y). 
	\end{equation} 
And therefore, by Boothby \cite[Theorem (3.1), p.\ 187]{wmB75}, $\xi$ determines the manifold topology on $\D$.

Define: 
    \begin{equation}  \label{E:G.is.group.of.diffeos}
      \text{Let } G \text{  be a finite group of diffeomorphisms on } \D
    \end{equation} 
and suppose that the Riemannian metric on $\D$ is $G$-invariant. I.e.,  
    \begin{equation}  \label{E:g.preserves.Riem.met}
        \text{for every } g \in G \text{ we have } 
          g^{\ast} \bigl( \langle \cdot, \cdot \rangle \bigr) 
            = \langle \cdot, \cdot \rangle. 
    \end{equation}
(A nontrivial $G$ will be important in chapters \ref{Chptr:spherical.location}, \ref{Chptr:aug.direct.mean}, and \ref{Chptr:robst.loc.on.circle}.) Such groups are important in Statistics. Indeed, de Finetti \cite[Chapter 11]{bdF2017.deFinetti.Thy.of.Prob} takes ``exchangeability'', invariance under a finite group, as a foundational concept in Statistics. We assume 
	\begin{equation*}
		g(\T) = \T \text{ for every } g \in G.
	\end{equation*}
Define
    \begin{equation} \label{E:G(x).G(A)}
        \text{for } x \in \D, \; G(x) := \{ g(x) : g \in G \} 
          \text{ and for } \mcl{A} \subset \D, \; G(\mcl{A}) := \{ g(y) : g \in G, y \in \mcl{A} \}.
    \end{equation}
By Boothby \cite[Theorem (1.2), p.\ 107]{wmB75},
	\begin{equation}  \label{E:g.ast.:Ty.to.Tg(y)}
		g_{\ast} : T_{y} \D \to T_{g(y)} \D, \quad g \in G, y \in \D.
	\end{equation}
	
 By \eqref{E:join.points.by.geod} and corollary \ref{C:geods.on.D.G-invar}, we have
	\begin{equation}   \label{E:xi.is.G-invar}
		x, y \in \D \text{ implies } \xi \bigl[ g(x), g(y) \bigr] = \xi(x,y) 
		     \text{ for every } g \in G.
	\end{equation}
I.e., each $g \in G$ is a an isometry on $\D$. 
Note that, by Boothby \cite[Exercise 6, p.\ 337]{wmB75}, 
each $g_{\ast} : T \D \to T \D$ is a diffeomorphism.

\section{Fiber bundles over $\Pf$ in $\D$ with cone fibers} \label{S:cone.bundles.over.P}
First, we show that we may assume that for some integer $k$, 
	\begin{multline}   \label{E:D.imbedded.in.Rk}
		\D \text{ is an imbedded submanifold of } \RR^{k} \\
		  \text{ with Riemannian metric induced by inclusion.}
	\end{multline} 

For $s \geq 0$ and subsets $A, B$ of a metric space $X$ with metric $m$, define the 
$\Hm^{s}$-essential distance from $A$ to $B$ as follows. Recall \eqref{E:set.distances}. Define
             \begin{equation}  \label{E:essential.dist.defn}
                dist^{s}(A,B) := dist_{m}^{s}(A,B)
                   := \sup \left\{ R \geq 0 : \Hm^{s} \Bigl( \bigl\{ x \in A : 
                       dist_{m}( x, B ) < R \bigr\} \Bigr) = 0 \right\} , 
             \end{equation} 
where, $\Hm^{s}$ is based on $m$. Let $s > 0$. Then $dist^{s}$ \emph{is not symmetric in its arguments.} To see this consider the case where 
$\Hm^{s}(A) = 0 < \Hm^{s}(B)$. Thus, $dist^{s}(A,B) = \infty$. 
On the other hand, 
$\Hm^{s} \Bigl( \bigl\{ x \in B : dist( x, A ) < 2 \, diam(\D) \bigr\} \Bigr) 
= \Hm^{s}(B) > 0$, 
so $dist^{s}(B,A) \leq 2 \, diam(\D) < \infty$. Note that, by \eqref{E:0.dim.Haus.measure}, for $s \geq 0$, 
    \begin{equation}  \label{E:dist.s,dist.0}
      dist^{s}(A,B) \geq dist^{0}(A,B) = dist(A,B) 
        := \inf \bigl\{ m(x,y) : x \in A, y \in B \bigr\},
    \end{equation}
the ordinary distance from $A$ to $B$. 

Let $\Phi : \D \partlyto \F$ be a data map. Let $\Pf$ be the set of perfect fits for the class of data maps to which $\Phi$ belongs and let $\Ss'$ be a closed superset of the singular set of $\Phi$. Let $d := \dim \D$, $p := \dim \Pf$. 
Let $a \geq 0$ be arbitrary. ($a$ will get meaning in property \ref{Pty:agree.near.T}.) 
In this chapter we will prove \eqref{E:Hm.a.S.geq.R.d-p-1}, which says there is an unspecified constant $\gamma > 0$ not depending on $\Phi$ s.t.\ 
             \begin{equation}  \label{E:Hm.a.S.geq.R.d-p-1.early}
                \text{If } R \leq dist_{\xi}^{a}(\Ss',\Pf) \text{ then }
                  \Hm^{a}(\Ss') \geq \gamma R^{ \min(d-p-1,a) } . 
             \end{equation} 
Here, we use $\xi$, \eqref{E:xi.is.metric.on.D}, to compute Hausdorff measure 
on $\D$. 

Actually, we will not prove \eqref{E:Hm.a.S.geq.R.d-p-1.early} directly. We will prove the \eqref{E:D.imbedded.in.Rk} version. We show that that version implies \eqref{E:Hm.a.S.geq.R.d-p-1.early}. By the Whitney imbedding theorem (Boothby \cite[Theorem (4.7), p.\ 195]{wmB75}) we may assume that for some integer $k$ there is an imbedding $f : \D \to \RR^{k}$. We replace $\D$ by $f(\D)$. 
Let  $i : f(\D) \hookrightarrow \RR^{k}$ be inclusion. Put on $f(\D)$ the Riemannian metric $i^{\ast}(dot)$, where $dot$ is the Riemannian metric on $\RR^{k}$, viz., Euclidean dot product. Let $\zeta$ be the topological metric on $f(\D)$ corresponding to $i^{\ast}(dot)$. In section \ref{SS:measure.proof} we will prove there exists $\gamma < \infty$ depending 
on $f(\D)$ but not on on $\Phi$ s.t.\
             \begin{equation}  \label{E:essntl.dist from.f(S).to.f(Pf).small.early}
                \text{If } R \leq dist_{\zeta}^{a} \bigl( f(\Ss'), f(\Pf) \bigr) \text{ then }
                  \Hm^{a} \bigl[ f(\Ss') \bigr] \geq \gamma R^{ \min(d-p-1,a) } , 
             \end{equation} 
where here $\Hm^{a}$ is computed using $\zeta$.

So assume \eqref{E:essntl.dist from.f(S).to.f(Pf).small.early} holds. From it we derive \eqref{E:Hm.a.S.geq.R.d-p-1.early}. Denote by $Euc$ the Euclidean metric 
on $\RR^{k}$: $Euc(x,x') := \|x-x'\|$ ($x,x' \in \RR^{k}$) and recall, from lemma \ref{L:imbedding.is.loc.Lip}, the meaning of $i^{\ast}(Euc)$. 
By Boothby \cite[Theorem (5.5), p.\ 78]{wmB75}, $f : \D \to f(\D)$ is a diffeomorphism. In particular, $f$ and its inverse are continuously differentiable. Therefore, by 
corollary \ref{C:cont.diff.=.loc.Lip} and \eqref{E:local.Lip.is.Lip.on.compacts} in \ref{Chptr:Lip.Haus.meas.dim}, the map $f$ and its inverse are Lipschitz on $\D$ and $f(\D)$, resp.\ relative to $\xi$ and $i^{\ast}(Euc)$. By lemma \ref{L:imbedding.is.loc.Lip} (with $M=N=f(\D)$) and \eqref{E:local.Lip.is.Lip.on.compacts} again, $f$ and its inverse are then Lipschitz 
on $\D$ and $f(\D)$, resp.\ (relative to $\xi$ and $\zeta$ on $f(\D)$). Let $\lambda \in (0, \infty)$ be a Lipschitz constant for both $f$ and its inverse (relative to $\xi$ and $\zeta$). 

Suppose $x \in \Ss'$ and $dist_{\xi}(x, \Pf) < R$. Then, by \eqref{E:set.distances}, there exists $z \in \Pf$ s.t.\ $\xi(x,z) < R$. Hence, $\bigl| f(x) - f(z) \bigr| < \lambda R$. It follows that $dist_{\zeta} \bigl( f(x), f(\Pf) \bigr) < \lambda R$. 

Now suppose $dist_{\xi}^{a}(\Ss', \Pf) \geq R$. Then, by \eqref{E:essential.dist.defn}, for every $n \in \NN$, 
$\Hm^{a} \Bigl( \bigl\{ x \in \Ss' : dist_{\xi}( x, \Pf ) < R-1/n \bigr\} \Bigr) = 0$. 
Now, if $y \in \RR^{k}$ and
$dist_{\zeta} \bigl( y, f(\Pf) \bigr) < \lambda^{-1} (R-1/n)$ then 
$dist_{\xi}\bigl( f^{-1}(y), \Pf \bigr) < R-1/n$. In particular, 
    \begin{equation*}
      f^{-1} \Bigl[ \bigl\{ y \in f(\Ss') : dist_{\zeta}( y, \Pf ) 
        < \lambda^{-1} (R-1/n) \bigr\}  \Bigr]
          \subset \bigl\{ x \in \Ss' : dist_{\xi}( x, \Pf )  < R-1/n \bigr\} .
    \end{equation*}
Hence, by lemma \ref{L:loc.Lip.image.of.null.set.is.null}, 
    \begin{multline*}
      0 = \Hm^{a} \Bigl( f \Bigl[ \bigl\{ x \in \Ss' : dist_{\xi}( x, \Pf ) 
        < R-1/n \bigr\} \Bigr] \Bigr) \\
          \geq \Hm^{a} \Bigl( \bigl\{ y \in f(\Ss')) : dist_{\zeta}( y, \Pf ) 
            < \lambda^{-1} (R-1/n) \bigr\} \Bigr) . 
    \end{multline*}
Taking the union over $n$ we get 
$\Hm^{a} \Bigl( \bigl\{ y \in f(\Ss') : dist_{\zeta}( x, \Pf ) < \lambda^{-1} R \bigr\} \Bigr) = 0$. 
Therefore, $dist_{\zeta}^{a} \bigl( f(\Ss'), f(\Pf) \bigr) \geq \lambda^{-1} R$. Thus, by \eqref{E:essntl.dist from.f(S).to.f(Pf).small.early},
      $\Hm^{a} \bigl[ f(\Ss') \bigr] \geq \gamma (\lambda^{-1} R)^{ \min(d-p-1,a) }$.

Now, by \eqref{E:Lip.magnification.of.Hm},  
    \begin{equation}  \label{E:Ha.S'.Ha.f(S').ineqs} 
       \Hm^{a} \bigl[ f(\Ss') \bigr] \leq \lambda^{a} \Hm^{a}(\Ss') .
    \end{equation}
Thus,  
    \begin{equation*}
      \Hm^{a}(\Ss') \geq \lambda^{-a} \Hm^{a} \bigl[ f(\Ss') \bigr] 
        \geq ( \lambda^{-a-1} \gamma) R^{ \min(d-p-1,a)} .
    \end{equation*}
I.e., by changing $\gamma$, \eqref{E:Hm.a.S.geq.R.d-p-1.early} holds 
in the $\D$ world. This proves \eqref{E:Hm.a.S.geq.R.d-p-1.early}.

In particular, we may assume an arbitrary tangent vector to any $x \in \D$ 
has the form $(x, v)$, where $v \in \RR^{k}$ with vector space operations defined by
    \begin{equation}  \label{E:vector.ops.on.TD}
        a (x, v) + b (x, w) = (x, av + bw) \text{ and } \bigl| (x,v) \bigr| = |v|, 
          \quad a, b \in \RR; v, w \in \RR^{k}; x \in \D.
    \end{equation}
The expression $a (x, v) + b (y, w)$ with $x, y \in \D$ distinct is not defined. Define
    \begin{equation}  \label{E:pi.is.proj}
      \pi(x, v) = x, \quad (x,v) \in T \D.
    \end{equation}
By Boothby \cite[Lemma (6.1), p.\ 332]{wmB75},
    \begin{equation}  \label{E:pi.is.smooth.and.open}
      \pi \text{ is } C^{\infty} \text{ and open}.
    \end{equation}

If we replace each $h \in G$ by $h_{f} := f \circ h \circ f^{-1}$ we get a group, isomorphic to $G$, of diffeomorphisms from $f(\D)$ to $f(\D)$. It remains to show the equivalent of \eqref{E:g.preserves.Riem.met}, i.e., that those maps preserve the Riemannian metric on $f(\D)$. If that is not true we can carry out the following manuever. Let $|G|$ be the cardinality of $G$ and define 
$f_{G} : \D \to \RR^{k \times |G|}$ by
	\begin{equation*}
		f_{G} : x \mapsto \bigl( f \circ g(x), \, g \in G \bigr) 
		  \in \RR^{k \times |G|}, \quad x \in \D.
	\end{equation*}
Thus, if we order the elements of $G$, $g_{1}, \ldots, g_{m}$, where $m := |G|$, then we can write $f_{G}(x) = \bigl( f \circ g_{1}(x) , \ldots, f \circ g_{m}(x) \bigr)$. 
So, by Boothby \cite[Theorem (5.7, p.\ 79]{wmB75}, $f_{G}$ is an imbedding of $\D$ into $\RR^{k \times |G|}$. For $h \in G$, define
$h_{f_{G}} := f_{G} \circ h \circ f_{G}^{-1}$ . Thus, $h_{f_{G}} : f_{G}(\D) \to f_{G}(\D)$, just permutes coordinates. Since the Riemannian metric 
on $\RR^{k \times |G|}$ is invariant under permutation of coordinates, we find 
that $\{ h_{f_{G}}, h \in G \}$ is a group of isometries on $f_{G}(\D)$. Now replace $k$ by $k \times |G|$ and $f$ by $f_{G}$ Then \eqref{E:D.imbedded.in.Rk} holds and, in this new setting, so does \eqref{E:g.preserves.Riem.met}. 

To recapitulate, we therefore may assume that 
	\begin{multline}  \label{E:Riem.metrics.on.D.on.Rk.same}
		\text{The Riemannian metric } \langle \cdot, \cdot \rangle_{x} \text{ on } \D 
		  \text{ is the one it inherits (pulled back) from } \\  
		    \RR^{k} = T_{x} \RR^{k}. \text{ In particular, } \| (x,v) \|_{x} = |v|.
		      \text{ Write } |(x,v)| = |v|, \quad (x, v) \in T_{x} \D.
	\end{multline} 
Here $\langle \cdot, \cdot \rangle_{x}$ is the Riemannian metric on $T_{x} \D$ and 
$\| \cdot \|_{x}$ is the corresponding norm. (Or we could use Nash imbedding, Han and Hong \cite{qHj-xH06.IsometricEmbeddings}.) Thus, we may assume that the Riemannian metric on $f(\D)$ is induced by inclusion.
 
 \subsection{Metrics on $\D$ and $T\D$} \label{SS:Metrics.on.D.and.TD}
 Note that $T \D$ is path connected by piece-wise $C^{1}$ paths:  
 Two points $(x,v), (x', v') \in T \D$ can be connected by a path as follows. 
 First, join $(x,v)$ to $(x,0)$ by a line segment in $T_{x} \D$. Then let $\gamma$ be a geodesic joining $x$ to $x'$ (exists by Boothby \cite[Corollary (7.11), p.\ 346]{wmB75}) and follow $(\gamma, 0)$ from $(x,0)$ to $(x',0)$. Then by a line segment 
 join $(x',0)$ to $(x',v')$.
 
 Let $\Gamma = (\gamma, V) : [0, t] \to T \D$ be a piece-wise $C^{1}$ 
 path in $T \D$. 
 Here, $\gamma : [0,t] \to \D$ and $V : [0,t] \to \RR^{k}$ 
 with $\bigl( \gamma(t), V(t) \bigr) \in T_{\gamma(t)} \D$. Define
     \begin{equation*}
        L_{\D}(\Gamma) 
          := \int_{0}^{t} \sqrt{ \bigl| \gamma'(s) \bigr|^{2} + \bigl| V'(s) \bigr|^{2} } \, ds
            = \int_{0}^{t} \bigl| \Gamma'(s) \bigr| \, ds,
    \end{equation*}
where $| \cdot |$ denotes Euclidean norm. Define a metric, $\omega_{\D}$, on $T \D$ by
    \begin{equation} \label{E:omega.D.defn}
      \omega_{\D} \bigl( (x, v), (x', v') \bigr) := \inf_{\Gamma} L_{\D}(\Gamma),
        \quad (x,v), (x',v') \in T \D,
    \end{equation}
where the infinum is taken over all piece-wise $C^{1}$ curves, 
$\Gamma : [0,t] \to T \D$ s.t.\ $\bigl( \gamma(0), V(0) \bigr) = (x, v)$ 
and $\bigl( \gamma(t), V(t) \bigr) = (x', v')$. 

Now let $b \geq 1$ and let $(x,v), (x', v') \in T \D$. Let $\Gamma = (\gamma, V) : [0,t] \to T \D$, piecewise $C^{1}$, join $(x,v) \in T \D$ 
to $(x',v') \in T \D$. Then $(\gamma, b^{-1} V)$ joins $(x,b^{-1}v)$ to $(x',b^{-1}v')$ and 
    \begin{align*}
        L_{\D}(\gamma, b^{-1} V)
          &= \int_{0}^{t} \sqrt{ \bigl| \gamma'(s) \bigr|^{2} + b^{-2} \bigl| V'(s) \bigr|^{2} } \, ds \\
              &\leq \int_{0}^{t} \sqrt{ \bigl| \gamma'(s) \bigr|^{2} + \bigl| V'(s) \bigr|^{2} } \, ds \\
                &= L_{\D}(\gamma, V).
    \end{align*}
It follows that 
    \begin{equation} \label{E:almost.homogeneity.of.omegaD}
       \omega_{\D} \bigl( (x, b^{-1} v), (x', b^{-1} v') \bigr) 
         \leq \omega_{\D} \bigl( (x, v), (x', v') \bigr),
           \quad (x,v), (x',v') \in T \D, b \geq 1.
    \end{equation}

The tangent bundle to a differentiable manifold is itself a differentiable manifold (Boothby \cite[Lemma (6.1), p.\ 332]{wmB75}). Hence, $T \RR^{k}$ and $T \D$ have tangent bundles (Fisher and Laquer \cite{rjFhtL1999.SecondOrderTangents}). Denote them by $TT \RR^{k}$ and $TT \D$, resp. A tangent vector to $(x, u) \in T \RR^{k}$ has the form $(x, u, v, w) \in TT \RR^{k} = \RR^{4k}$, where $(x,v) \in T_{x} \RR^{2k}$ and $w \in \RR^{k}$ can be thought of as tangent to $u$. Put on $TT \RR^{k}$ the bilinear form, 
    \begin{multline} \label{E:Riem.metric.on.T.R^k}
      \bigl\langle (x, u, v, w), (x, u, v', w') \bigr\rangle := v \cdot v' + w \cdot w', \\
        (x, u, v, w), (x, u, v', w') \in T_{(x,u)} T \D . 
    \end{multline} 
This form is obviously symmetric, positive definite, and $C^{\infty}$ so turns $T \D$ into a Riemannian manifold (Boothby \cite[Definition (2.6), p.\ 184]{wmB75}). By \eqref{E:D.imbedded.in.Rk}, $\D$ is an imbedded submanifold of $\RR^{k}$. Let $inc : \D \to \RR^{k}$ be the (smooth) imbedding. 

By (Boothby \cite[Exercise 6, p.\ 337]{wmB75}), $inc_{\ast} : T \D \to T \RR^{k}$ is smooth. 
By Boothby \cite[Theorem (5.5), p.\ 201]{wmB75}, 
$(inc_{\ast})^{\ast} \bigl( \langle \cdot \rangle \bigr)$, where $\langle \cdot \rangle$ is defined in \eqref{E:Riem.metric.on.T.R^k}, is a $C^{\infty}$ bilnear form on $\D$. It is also obviously symmetric and positive semidefinite. In fact, $(inc_{\ast})^{\ast} \bigl( \langle \cdot \rangle \bigr)$ is actually positive definite on $T \D$ making $T \D$ a Riemannian manifold. I.e., 
    \begin{equation} \label{E:Riem.metric.on.TD}
      \eqref{E:Riem.metric.on.T.R^k} \text{ defines a Riemannian metric on } T \D. 
    \end{equation}
See appendix \ref{Chptr:misc.proofs} for the rest of the proof of \eqref{E:Riem.metric.on.TD}.

If $\Gamma = (\gamma, V) : [0, t] \to T \D$ is a piece-wise $C^{1}$ path in $T \D$, where 
$\gamma : [0,t] \to \D$ and $V : [0,t] \to \RR^{k}$ 
 with $\bigl( \gamma(t), V(t) \bigr) \in T_{\gamma(t)} \D$ and $s \in (0,t)$, then the tangent vector to $\Gamma$ at $\Gamma(s)$ is  
$\bigl( \gamma(s), V(s), \gamma'(s), V'(s) \bigr) \in T_{\bigl(\gamma(s), V(s)\bigr)}TD$. 
Thus, $\omega_{\D}$ is just the topological metric on $T \D$ corresponding to this Riemannian metric and, by Boothby \cite[Corollary (7.5), p.\ 342]{wmB75}, a path in $TD$ of minimal $L_{\D}$-length is a geodesic in $T \D$. 
  
 For $b \geq 1$, let 
    \begin{equation*}
       T^{b} \D := \bigl\{ (x,v) \in T \D : 
         |v| \leq b \bigr\} \subset \D \times \overline{B_{b}^{k}(0)}.
    \end{equation*}
(See \eqref{E:Euc.ball.defn}.) 
 Since $\D \subset \RR^{k}$ is compact, by \eqref{E:D.is.cmpct.Riem.manif}, we have that $T^{b} \D$ is a compact subset of $T \D$. Let $g \in G$. Then, by \eqref{E:g.preserves.Riem.met}, we have $g_{\ast} : T^{b} \D \to T^{b} \D$. Recall the definition of ``bi-Lipschitz'', \eqref{E:bi-Lipschitz.defn}. 
\emph{Claim:} There exists $K < \infty$ s.t.\ for every $b \geq 1$ we have
    \begin{equation} \label{E:g*.Lip.const.Kb.omegaD}
        g_{\ast} \text{ is bi-Lipschitz on } T^{b} \D \text{ w.r.t.\ } \omega_{\D} 
          \text{ with Lipschitz constant } K b.
    \end{equation}
Now, $g_{\ast} : T \D \to T \D$ is $C^{\infty}$ (Boothby \cite[Exercise 6, p.\ 337]{wmB75}). Therefore, by corollary \ref{C:cont.diff.=.loc.Lip} there exists $K < \infty$ s.t.\ $g_{\ast}$ is bi-Lipschitz on $T^{1} \D$ with Lipschitz constant $K < \infty$, uniformly in $g \in G$.
 
Let $(x,v), (x',v') \in T^{b} \D$ so $(x, b^{-1} v), (x', b^{-1} v'), g_{\ast}(x, b^{-1}v), g_{\ast}(x',b^{-1} v') \in T^{1} \D$. Let $\epsilon > 0$ be arbitrary  and let $\Gamma = (\gamma, V) : [0,t] \to T \D$ be a piecewise $C^{1}$ path joining 
 $g_{\ast}(x, b^{-1}v)$ to $g_{\ast}(x',b^{-1} v')$ s.t.\ 
    \begin{equation*}
        L_{\D}(\Gamma) \leq \omega_{\D} \bigl[ g_{\ast}(x, b^{-1}v), g_{\ast}(x',b^{-1} v') \bigr] + \epsilon.
    \end{equation*}
Thus, $(\gamma, bV)$ joins $g_{\ast}(x, v)$ to $g_{\ast}(x', v')$ Then, since $b \geq 1$ and using \eqref{E:almost.homogeneity.of.omegaD},
    \begin{align*}
        \omega_{\D} \bigl[ g_{\ast}(x,v), g_{\ast}(x',v') \bigr] 
          &\leq L_{\D} \bigl[ (\gamma, b V) \bigr]
         = \int_{0}^{t} \sqrt{ \bigl| \gamma'(s) 
           \bigr|^{2} + b^{2} \bigl| V'(s) \bigr|^{2} } \, ds \\
           &\leq b \int_{0}^{t} \sqrt{ \bigl| \gamma'(s) 
             \bigr|^{2} + \bigl| V'(s) \bigr|^{2} } \, ds \\
             &= b L_{\D} (\Gamma) 
               \leq b \, \omega_{\D} \bigl[ g_{\ast}(x, b^{-1}v), g_{\ast}(x',b^{-1} v') \bigr] 
                 + b \epsilon \\
                 &\leq b K \omega_{\D} \bigl[ (x, b^{-1}v), (x',b^{-1} v') \bigr] + b \epsilon \\
                   &\leq b K \omega_{\D} \bigl[ (x, v), (x',v') \bigr] + b \epsilon.
    \end{align*}
Letting $\epsilon \downarrow 0$, we see that $g_{\ast}$ has Lipschitz constant $b K$ on $T^{b} \D$. This completes the proof of the claim \eqref{E:g*.Lip.const.Kb.omegaD}. (See \eqref{E:g*.Lip.const.Kb.sqrd.xi+}.)

Recall the metric $\xi$ defined in \eqref{E:xi.is.metric.on.D}. Define another metric on $T \D$ by
          \begin{equation}  \label{E:xi+.from.2.metrics}
        		\xi_{+} \bigl( (x, v), (x', v') \bigr) := \sqrt{ \xi(x, x')^{2} + |v - v'|^{2} }, 
        			\quad (x,v), (x',v') \in T \D.
        	\end{equation} 
Note that both the maps $\pi : (x,v) \mapsto x$ and 
$v \mapsto \bigl| (x,v) \bigr| :=  |v|$, with $(x,v) \in T \D$, are continuous w.r.t.\ the metric $\xi_{+}$. Identify $x$ and $(x, 0)$ ($x \in \D$). Then we have 
   \begin{multline}  \label{E:xi+.xi.and.|.|}
      \xi_{+}(x', x'') = \xi(x', x'') \text{ and }
        \xi_{+} \bigl[ (x',v), (x',0) \bigr] = |v| \\
           \text{ for every } x'' \in \Pf \text{ and } (x',v) \in T \D.
   \end{multline}

Now, by \eqref{E:n.c.sqrd.sum.ineq} we have 
    \begin{equation*}
      \sqrt{2} \sqrt{ \bigl| \gamma'(s) \bigr|^{2} + \bigl| V'(s) \bigr|^{2} } 
        \geq \bigl| \gamma'(s) \bigr| + \bigl| V'(s) \bigr|.
    \end{equation*}
Moreover, a path $V(s)$ ($s \in [0,t]$) joining $v \in \RR^{k}$ to $v' \in \RR^{k}$ can be no shorter than a linear path. Thus,
    \begin{equation*}
        \int_{0}^{t} \bigl| V'(s) \bigr| \, ds \geq |v - v'|.
    \end{equation*}
Therefore, by the triangle inequality,
    \begin{align} \label{E:sqrt2.omegaD.geq.xi+}
     \sqrt{2} \, \omega_{\D} \bigl( (x, v), (x', v') \bigr)
       &= \sqrt{2} \, \inf_{\gamma,V} \int_{0}^{t} \sqrt{ \bigl| \gamma'(s) \bigr|^{2} 
         + \bigl| V'(s) \bigr|^{2} } \, ds  \notag \\
       &\geq \text{(length of } \gamma) + |v - v'| \\
       &\geq \xi(x, x') + |v-v'| \geq \sqrt{ \xi(x, x')^{2} + |v - v'|^{2} } \notag \\
       &= \xi_{+} \bigl( (x, v), (x', v') \bigr). \notag
    \end{align}

\emph{Claim:} We have the following partial converse to \eqref{E:sqrt2.omegaD.geq.xi+}.
There exists $K < \infty$ s.t.\ 
    \begin{equation} \label{E:omegaD.leq.Kb.xi+}
        \omega_{D} \bigl[ (x,v), (x',v') \bigr] \leq K b \, \xi_{+} \bigl[ (x,v), (x',v') \bigr],
          \quad b \geq 1; (x,v), (x',v') \in T^{b} \D.
    \end{equation}
\eqref{E:sqrt2.omegaD.geq.xi+} and \eqref{E:omegaD.leq.Kb.xi+} imply that $\omega_{\D}$ and $\xi_{+}$ determine the same topology on $T \D$ and, in fact,
    \begin{multline} \label{E:omegaD.xi+.Lip.property}
      \text{ The identity map on $T \D$ is Lipschitz when viewed as a map from } \\
        \bigl( T \D, \omega_{\D} \bigr) \text{ to } \bigl( T \D, \xi_{+} \bigr)
          \text{ and locally Lipschitz in the other direction }
    \end{multline}

The proof of \eqref{E:omegaD.leq.Kb.xi+} is very similar to that of \eqref{E:g*.Lip.const.Kb.omegaD}. Let $M := \D \times \RR^{k}$ and 
$N := T \D$ so $N \subset M$. Let $f : N \to M$ be inclusion. Then $\xi_{+}$ comes from the Riemannian metric on $M$ 
and $\omega_{\D}$ comes from the Riemannian metric on $N$. Then by lemma \ref{L:imbedding.is.loc.Lip}, there exists $K < \infty$ 
s.t.\ on $T^{1} \D$ we have $\omega_{D} \leq K \xi_{+}$. Now let $b \geq 1$ 
and let $(x,v), (x',v') \in T^{b} \D$. Let $\epsilon > 0$ be arbitrary 
and let $\Gamma = (\gamma, V) : [0,t] \to T \D$ be a path joining 
$(x, b^{-1}v), (x',b^{-1} v') \in T^{1} \D$ s.t.\ 
    \begin{equation*}
        L_{\D}(\Gamma) \leq \omega_{\D} \bigl[ (x, b^{-1}v), (x',b^{-1} v') \bigr] + \epsilon.
    \end{equation*}
Notice
    \begin{multline*}
        b \xi_{+} \bigl[ (x, b^{-1}v), (x',b^{-1} v') \bigr] 
          = \sqrt{ b^{2} \xi(x, x')^{2} + |v-v'|^{2} } \\
            \leq b \sqrt{ \xi(x, x')^{2} + |v-v'|^{2} } 
              = b \xi_{+} \bigl[ (x, v), (x',v') \bigr].
    \end{multline*}
I.e.,
    \begin{equation*}
     \xi_{+} \bigl[ (x, b^{-1}v), (x',b^{-1} v') \bigr]  \leq \xi_{+} \bigl[ (x, v), (x',v') \bigr]
      \text{ for } b \geq 1.
    \end{equation*}

We have
    \begin{align*}
    \omega_{\D} \bigl[ (x,v), (x',v') \bigr] &\leq L_{\D} \bigl[ (\gamma, b V) \bigr]
     = \int_{0}^{t} \sqrt{ \bigl| \gamma'(s) \bigr|^{2} + b^{2} \bigl| V'(s) \bigr|^{2} } \, ds \\
       &\leq b \int_{0}^{t} \sqrt{ \bigl| \gamma'(s) \bigr|^{2} + \bigl| V'(s) \bigr|^{2} } \, ds \\
         &= b L_{\D} (\Gamma) 
           \leq b \omega_{\D} \bigl[ (x, b^{-1}v), (x',b^{-1} v') \bigr] + b \epsilon \\
             &\leq b K \xi_{+} \bigl[ (x, b^{-1}v), (x',b^{-1} v') \bigr] + b \epsilon \\
               &\leq b K \xi_{+} \bigl[ (x,v), (x',v') \bigr] + b \epsilon.
    \end{align*}
Letting $\epsilon \downarrow 0$ the proof of the claim \eqref{E:omegaD.leq.Kb.xi+} is completed.

See appendix \ref{Chptr:misc.proofs} for the proof of the following.
  \begin{lemma} \label{L:xi+.generates.manif.topol}    
$\xi_{+}$, and hence $\omega_{\D}$, generates the standard topology 
on $T \D$.  
  \end{lemma}
  
Recall the definition of ``bi-Lipschitz'', \eqref{E:bi-Lipschitz.defn}. Combining \eqref{E:g*.Lip.const.Kb.omegaD}, \eqref{E:sqrt2.omegaD.geq.xi+},  \eqref{E:omegaD.leq.Kb.xi+} and the fact that $G$ is a group, we get there exists 
$K < \infty$ s.t.\ 
    \begin{equation} \label{E:g*.Lip.const.Kb.sqrd.xi+}
        \text{For every } b \geq 1,  \, g_{\ast} \text{ is bi-Lipschitz on } T^{b} \D 
          \text{ w.r.t.\ } \xi_{+} \text{ with Lipschitz constant } K b^{2}.
    \end{equation}
Recall, \eqref{E:vector.ops.on.TD}, how scalar multiplication works on $T \D$. It is easy to see
  \begin{lemma} \label{L:scalar.mult.xi+.cont}
 Let $\X \subset \D$ and let $T \D \restriction_{\X} := \bigl\{ (x,v) \in T \D : x \in \X \bigr\}$ be the restriction of $T \D$ to $\X$ 
(Milnor and Stasheff \cite[p.\ 25]{jwMjdS74}). Let 
$r : \X \to \RR$ be continuous. Then the map $f : (x,v) \mapsto r(x) (x,v) = \bigl(x, r(x) v \bigr)$ is continuous w.r.t.\ $\xi_{+}$. Hence. if $r$ is nowhere vanishing $f$ is a heomomorphism of $T \D \restriction_{\X}$ onto itself.
  \end{lemma}
  
 Let $Exp$ be the exponential map on $T \D$ (Boothby \cite[Definition (6.3), p.\ 333]{wmB75}). Since $\D$ is compact (by \eqref{E:D.is.cmpct.Riem.manif}), by Hopf-Rinow (Boothby \cite[Theorem (7.7), p.\ 343]{wmB75}), $Exp$ is defined on all of the tangent bundle $T \D$. We have the following. See appendix \ref{Chptr:misc.proofs} for the proof.

  \begin{lemma} \label{L:Exp.loc.Lip}
    $Exp$ is locally Lipschitz on $T \D$ w.r.t.\ $\xi_{+}$ and $\xi$, hence, 
      by \eqref{E:sqrt2.omegaD.geq.xi+}, \emph{a fortiori} w.r.t.\ $\omega_{\D}$ 
      and $\xi$.
  \end{lemma}

\subsection{Tubular neighborhood}  \label{SSS:tubular.nbhd}
First, we consider the case in which the ``perfect fit space'', $\Pf$, is a submanifold 
of $\D$. 
Assume $p := \dim \Pf < d := \dim \D$. ($p=0$ is possible.) In this book we employ the following notation. If $w$ is a vector and $W$ is a linear subspace of an inner product space $V$ (with inner product $\langle \cdot, \cdot \rangle$), then 
    \begin{equation} \label{E:superscript.perp.notation}
      w^{\perp} := \bigl\{ v \in V : \langle v, w \rangle = 0 \bigr\} \text{ and }
        W^{\perp} := \bigcap_{w \in W} w^{\perp} 
          = \{ v \in V : v \perp W \} .
    \end{equation}
Thus, $W^{\perp}$ is the orthogonal complement of $W$ in $V$. 

If $x' \in \Pf$, let $T_{x'}\D$ ($T_{x'} \Pf$) be the tangent space to $\D$ (resp.\ $\Pf$) at $x'$ and let 
$(T_{x'} \Pf)^{\perp} = \bigl\{ (x',v) \in T_{x'} \D : v \perp T_{x'} \Pf \bigr\}$ denote the subspace of $T_{x'} \D$ normal to $\Pf$. 
Thus, $\dim \bigl[ (T_{x'} \Pf)^{\perp} \bigr] = d-p$. Let 
	\[
		N = N(\Pf, \D) = \bigcup_{x' \in \Pf} (T_{x'} \Pf)^{\perp}
		         \subset \Pf \times \RR^{k}
	\]
be the (total space of the) normal bundle of $\Pf$ in $\D$ (Milnor and Stasheff \cite[ p.\ 29]{jwMjdS74}). 
(In this chapter we generally use $\mcl{CALLIGRAPHIC}$ symbols for subsets of $\D$ and $ORDINARY$ symbols for subsets of $T \D$, among other things. I admit that sometimes the two kinds of symbols and look quite similar. Sorry.) The ``zero section'' of $N$ is the set of points $\bigl\{ (x', 0) \in N: x' \in \Pf \bigr\}$.

We use the tubular neighborhood theorem stated, e.g., in Milnor and Stasheff \cite[Theorem 11.1, p.\ 115]{jwMjdS74} or Guillemin and Pollack \cite[Exercise 16, p.\ 76]{vGaP74.DiffTopol}; see also Spivak \cite[Theorem 20, p.\ 467]{mS79.SpivakVol1}. Unfortunately, these versions do not give us exactly what we want. So we assemble our own version. (See appendix \ref{Chptr:misc.proofs} for proof.) Define ``$dist$'' in $\D$ (see \eqref{E:set.distances}) using $\xi$.

  \begin{prop}[Tubular Neighborhood Theorem]  \label{P:tubular.nbhd.thm}
Let $\Pf$ be a smooth imbedded submanifold of $\D$ (with $p < d$). $\Pf$ does not have to be compact. The total space of the normal bundle $N(\Pf, \D)$ is an immersed $d$-dimensional differentiable submanifold of $T \D$.
If $\epsilon : \Pf \to (0, \infty]$ is a positive function, define 
    \begin{equation}  \label{E:tubular.N.eps.defn}
		\hat{N}^{\epsilon} 
		     := \Bigl\{ (x',v) \in N(\Pf, \D) : |v| < \epsilon(x') \Bigr\}.    
    \end{equation}
There exists a positive smooth function $\epsilon_{\Pf} : \Pf \to (0, \infty)$ s.t. 
$\hat{N}^{\epsilon_{\Pf}}$ is an immersed submanifold of $T \D$. Moreover, $Exp$ is defined on $\hat{N}^{\epsilon_{\Pf}}$ and maps it diffeomorphically onto a neighborhood $\mcl{C}$ of $\Pf$. (Call $\mcl{C}$ a ``tubular neighborhood of $\Pf$.)

Let $\alpha := (Exp \restriction_{\hat{N}^{\epsilon_{\Pf}}})^{-1}$, so $\alpha$ is a diffeomorphism of $\mcl{C}$ onto 
$\hat{N}^{\epsilon_{\Pf}}$. If $(x', v) \in \hat{N}^{\epsilon_{\Pf}}$, then we may assume
	\begin{equation} \label{E:alpha.dist.to.P} 
			\text{the $\xi$-closest point of $\Pf$ to } Exp(x',v) 
			        \text{ is } x' \text{ and } \text{dist} \bigl[ Exp(x',v), \Pf \bigr] = |v|.
	\end{equation}
  \end{prop}
  
As usual, let $G$ be a finite group of diffeomorphisms mapping $\D$ into itself. If $(x', v)$ is \emph{any} vector in $T \D$ and $g \in G$ then  
	\begin{equation}  \label{E:gExp=Expg.star}
		g \circ Exp(x',v) = Exp \circ g_{\ast}(x',v)  
		         \text{ and } \bigl| g_{\ast}(x',v) \bigr| = |v|,
	\end{equation}
where $g_{\ast} : T \D \to T \D$ is the differential of $g$. (For proof see appendix \ref{Chptr:misc.proofs}.) 

Suppose that 
    \begin{equation}  \label{E:P.is.G-invar}
      \Pf \text{ is $G$-invariant. I.e., } g(\Pf) = \Pf \text{ for every } g \in G.
    \end{equation} 
Since the Riemannian metric on $\D$ is $G$-invariant, by \eqref{E:g.preserves.Riem.met}, if $(x', v) \in N$ and $g \in G$ then $g_{\ast}(x', v) \perp T_{g(x')} \Pf$ and, by \eqref{E:gExp=Expg.star} and \eqref{E:xi.is.G-invar}, 
	\begin{multline}   \label{E:tubular.nbhd.of.g.C}
		\text{For every } g \in G \text{ we have } g(\mcl{C})
		        \text{ is a tubular neighborhood of } \Pf \\
		   \text{ and } \alpha \circ g^{-1} : g(\mcl{C}) \to \hat{N}^{\epsilon_{\Pf}} 
		           \text{ is a diffeomorphism.}
	\end{multline} 
Let $(x', v) \in N$ and $g \in G$. Write $g_{\ast}(x',v) = \bigl( g(x'), w \bigr)$, 
so $\bigl( g(x'), w \bigr) \perp T_{g(x')} \, \Pf$. Since $g_{\ast}$ is linear, if $t \in \RR$ then, 
$g_{\ast}(x', tv) = \bigl( g(x'),  t w \bigr)$. (See \eqref{E:vector.ops.on.TD}.)

Since $Exp(x',0) = x' \in \Pf$, we have 
	\begin{equation} \label{E:alpha.(x',0).=.x'}
		\alpha(x') = (x', 0) \text{ for every } x' \in \Pf.
	\end{equation}
We identify $\Pf$ and $\Pf \times \{ 0 \} \subset \hat{N}^{\epsilon_{\Pf}}$, 
so $\alpha \restriction_{\Pf}$ is the identity on $\Pf$. 
Let $\pi : \hat{N}^{\epsilon_{\Pf}} \to \Pf$ be the projection map 
in $\hat{N}^{\epsilon_{\Pf}}$. 
I.e., $\pi(x,v) = x$ if $(x,v) \in \hat{N}^{\epsilon_{\Pf}}$. Thus,
	\begin{equation*}
		\alpha : \mcl{C} \to \hat{N}^{\epsilon_{\Pf}} 
		  \text{ and } \pi : \hat{N}^{\epsilon_{\Pf}} \to \Pf.
	\end{equation*}
	
$G$ acts on $T \D$ by  
	\begin{equation}  \label{E:G.acts.on.TD}
		g(x', v) := g_{\ast}(x', v) \in T_{g(x)} \D    
		          \quad \bigl( (x',v) \in T \D, \, g \in G \bigr). 
	\end{equation}
Therefore, if $(x', v) \in \hat{N}^{\epsilon_{\Pf}}$ 
and $g_{\ast}(x', v) \in \hat{N}^{\epsilon_{\Pf}}$ (in particular $g(x') \in \Pf$), we have
	\begin{equation}  \label{E:pi.is.G.equivar}
		\pi \circ g_{\ast}(x', v) = g \circ \pi(x',v).
	\end{equation}
  
By \eqref{E:P.is.G-invar} and \eqref{E:gExp=Expg.star}, 
	\begin{equation*}
	  g(\mcl{C}) = g \bigl( Exp(\hat{N}^{\epsilon_{\Pf}}) \bigr) 
	    = Exp(\hat{N}^{\epsilon_{\Pf} \circ g^{-1}}) \quad (g \in G). 
	\end{equation*}
Thus, if $\epsilon_{\Pf}$ is not $G$-invariant ($G$-invariant means 
$\epsilon_{\Pf} \circ g = \epsilon_{\Pf}$ for every $g \in G$), 
then we might replace $\epsilon_{\Pf}$ by $\min \{ \epsilon_{\Pf} \circ g, \, g \in G \}$. Then \eqref{E:alpha.dist.to.P} still holds and   
	    \begin{equation}  \label{E:tube.nbhd.G-invar}
		g(\mcl{C}) = \mcl{C} \text{ and } 
		  g_{\ast} (y,v) \in (T_{g(y)} \Pf)^{\perp}
		    \text{ for every } y \in \Pf, (y,v) \in (T_{y} \Pf)^{\perp}, 
		      \text{ and } g \in G.
	    \end{equation} 

However, $\min \{ \epsilon_{\Pf} \circ g, \, g \in G \}$ usually will not be smooth. However, we have the following. See appendix \ref{Chptr:misc.proofs} for the proof.
    \begin{multline}  \label{E:smooth.<.min}
        \text{There is a smooth positive function } \epsilon : \Pf \to (0, \infty) \text{ s.t.\ } \\
          \epsilon(x) \leq \min \{ \epsilon_{\Pf} \circ g(x), \, g \in G \} \text{ for every }
            x \in \Pf \text{ and } \epsilon \circ g = \epsilon, \; g \in G.
    \end{multline}

Now we move on to consider more general $\Pf$ and subspaces of $T_{x} \D$ ($x \in \Pf$) more complex than linear subspaces.

\subsection{Conical fibers} \label{SSS:conical.fibers} 
If the group $G$ is trivial (as in chapter \ref{Chptr:sings.in.plane.fit}) or if $\Pf = \T$ then we might get away with having $\Pf$ being a manifold. However, if $G$ is nontrivial and $\Pf \neq \T$ then $\Pf$ may not be a manifold and we cannot use the Tubular Neighborhood Theorem stated in the last section. (This is the situation in chapter \ref{Chptr:robst.loc.on.circle}.) So more generally, we assume that 
    \begin{equation}  \label{E:Pf.loc.compct.strat.space}
        \Pf \text{ is a locally compact stratified space. }
    \end{equation}

In this book ``stratified space" will mean a non-empty subspace, $X$, of a Riemannian manifold $M$ expressed as the \emph{disjoint} union of finitely many ``strata'', i.e., connected imbedded smooth submanifolds, $Y$, of $M$. Call the list of strata a ``stratification'' of $X$. More precisely the stratified space consists of $X$, $M$, and the stratification. However, we will just refer to $X$ as a stratified space with the understanding that there is a stratification. If $p = 0, 1, \ldots$ is the highest dimension of any stratum in $X$, then we call $p$ the ``dimension of $X$'' and write $\dim X = p$. (In fact, by  \eqref{E:Haus.dim.s-manif.=.s} and \eqref{E:dim.of.whole.=.max.dim.of.parts}, $p$ is just the Hausdorff dimension of $X$.) 

Since a stratified space is the union of finitely many connected manifolds we have
    \begin{equation}  \label{E:0.dim.stratified.spaces}
      \text{if } X \text{ is a stratified space of dimension 0 then $X$ is finite.}
    \end{equation}

If $K$ is a compact subset of $X$ and $Y$ is a stratum it may not be the case that 
$K \cap Y$ is compact. However, we sometimes assume the following tameness property. Call a coordinate neighborhood, $(U, \varphi)$, of $Y$ ``bi-Lipschitz'' if its coordinate map 
$\varphi : U \to \RR^{J}$ (for some $J$; $\varphi$ is smooth) is bi-Lipschitz (w.r.t.\ the topological metric on $Y$ induced by the Riemannian metric on $M$; $\varphi$ is smooth in both directions, of course; see definition of ``bi-Lipschitz:'' \eqref{E:bi-Lipschitz.defn}). Say that $(U, \varphi)$ has ``bounded convex parameter space" if $\varphi(U) \subset \RR^{J}$ is bounded and convex. Say that $(U, \varphi)$ is ``tractable'' if it is bi-Lipschitz and has bounded convex parameter space. If $\dim Y = 0$, then a coordinate neighborhood $(U, \varphi)$ is tractable if $U$ consists of a single point. (So $\varphi$ is constant.) Say that $U \subset Y$ is tractable 
if there exists $\varphi$ s.t.\ $(U, \varphi)$ is a tractable coordinate neighborhood. Then we sometimes assume the following tameness property.
    \begin{multline}  \label{E:tameness.of.stratified.space}
        \text{If } K \subset X \text{ is compact and } Y \text{ is any stratum of } X \\ 
          \text{ then } K \cap Y  
              \text{ is covered by finitely many tractable coordinate neighborhoods of } Y. 
    \end{multline}
(I don't know how this condition is related to the Whitney conditions, Mather \cite{jM2012.MatherStratSpaces}, on stratified spaces.) Since a single point is compact, it follows that 
    \begin{multline*} 
      \text{If $X$ satisfies \eqref{E:tameness.of.stratified.space} } \\
        \text{ then each stratum $Y$ is covered by tractable coordinate neighborhoods.}
    \end{multline*}

Note that if $X$ satisfies \eqref{E:tameness.of.stratified.space}, it automatically satisfies this seemingly stronger version:
    \begin{multline}  \label{E:tameness.of.stratified.space.small}
        \text{If } K \subset X \text{ is compact and } Y \text{ is a stratum of } X  
          \text{ of positive dimension then } K \cap Y \\           
              \text{ is covered by finitely many tractable coordinate neighborhoods of } Y \\
                \text{ of aribitrarily small diameter}. 
    \end{multline}
To prove this it suffices to show that any tractable neighborhood in a $J$-dimensional manifold $Y$ is covered by finitely many small tractable neighborhoods. Let $(U, \varphi)$ be a tractable coordinate neighborhood of $Y$. Let $\epsilon > 0$ be given. Let $K < \infty$ be a Lipschitz constant 
for $\varphi^{-1}$ and let $\delta < \epsilon/K$. 
Let $V := \varphi(U) \subset \RR^{J}$. 
Since $(U, \varphi)$ is tractable, $V$ is convex and bounded. Hence, it can be covered by finitely many balls of radius $\delta$. Let $B$ be such a ball, so $B \cap V \neq \varnothing$. Since $V$ is convex so is $B \cap V$. Let $U_{B} := \varphi^{-1}(B \cap V)$. Let $\varphi_{B}$ be the restriction $\varphi \restriction_{U_{B}}$. Then $(U_{B}, \varphi_{B})$ is a tractable coordinate neighborhood. By choice of $\delta$, $diam \, U_{B} < \epsilon$. This proves \eqref{E:tameness.of.stratified.space.small}.

Our version of ``stratified space'' differs from others'. (See Pflaum \cite{mjP.StratSpaces} and Banagl \cite{mB07.StratSpaces}.) Note that a smooth manifold is a stratified space with just one stratum. By corollary \ref{C:cont.diff.=.loc.Lip}, it has property \eqref{E:tameness.of.stratified.space}. 

Note that if $X$ is compact, then we may take $K=X$ in \eqref{E:tameness.of.stratified.space}. A  smooth manifold is ``finitely tractable'' if it is covered by finitely many tractable neighborhoods. 

  \begin{lemma}  \label{L:strat.properties} 
    \begin{enumerate}
\item A smooth manifold is covered by tractable neighborhoods. 
    \label{I:manif.covered.by.tract.nbhds}
\item Any single stratum stratified space satisfies \eqref{E:tameness.of.stratified.space}. 
    \label{I:single.stratum.space.is.tame}
\item A compact smooth manifold is finitely tractable. 
    \label{I:cmpct.manif.finitely.tractble}
\item If every stratum of $X$ is finitely tractable then $X$ satisfies \eqref{E:tameness.of.stratified.space}. 
    \label{I:all.strata.finitely.trctble.means.tame}
\item Let $X$ have a $p$-dimensional stratification. Then \emph{any} stratification of $X$ is \linebreak $p$-dimensional. \label{I:dim.of.strat.space.is.same.for.every.stratn}
\item A compact $p$-dimensional stratified space whose $\Hm^{p}$-dimensional measure is infinite cannot satisfy \eqref{E:tameness.of.stratified.space}. 
   \label{I:cmpct.strat.space.with.infinite.measure.is.not.tame}
    \end{enumerate}
  \end{lemma}
  \begin{proof}[Proof of lemma \ref{L:strat.properties}]
 We prove \ref{I:manif.covered.by.tract.nbhds}. Suppose $X$ is a $m$-dimensional smooth manifold, not necessarily compact. We show that $X$ is covered by tractable neighborhoods. For $x \in X$, 
let $(\clU_{x}, \varphi_{x})$ be a coordinate neighborhood containing $x$ and let $y := \varphi_{x}(x) \in \RR^{m}$. Let $B_{r}(y)$ be an open ball about $y$ of radius $r > 0$ s.t.\ $B_{r}(y)\subset \varphi(\clU_{x})$. Thus, the closure, $\overline{B_{r/2}(y)}$, is a compact subset of $\varphi(\clU_{x})$. Let $\mcl{V}_{x} = \varphi^{-1} \bigl( B_{r/2}(y) \bigr)$ (so $x \in \mcl{V}_{x}$) and 
let $\phi_{x} : \mcl{V}_{x} \to \RR^{m}$ be the restriction $\varphi \restriction_{\mcl{V}_{x}}$. Then by corollary \ref{C:cont.diff.=.loc.Lip}, 
$\phi_{x}$ and its inverse are Lipschitz. $X$ is covered by the coordinate neighborhoods, $(\mcl{V}_{x}, \phi_{x})$, each parametrized by a bounded convex set. 

Suppose $X$ has just one stratum, itself. Thus, $X$ is a smooth manifold. By part \eqref{I:manif.covered.by.tract.nbhds}, $X$ is covered by tractable neighborhoods. Hence, any compact subset of $X$ is covered by finitely many tractable neighborhoods. If follows that $X$ satisfies \eqref{E:tameness.of.stratified.space}. This proves part \ref{I:single.stratum.space.is.tame}.

Statements \ref{I:all.strata.finitely.trctble.means.tame} and, given \ref{I:manif.covered.by.tract.nbhds}, \ref{I:cmpct.manif.finitely.tractble} are obvious. Part \ref{I:dim.of.strat.space.is.same.for.every.stratn} is immediate from \eqref{E:dim.of.whole.=.max.dim.of.parts} and \eqref{E:Haus.dim.s-manif.=.s}. 

We prove part \ref{I:cmpct.strat.space.with.infinite.measure.is.not.tame}. Let $X$ be a compact $p$-dimensional stratified space whose $\Hm^{p}$-dimensional measure is infinite. Suppose $X$ satisfies \eqref{E:tameness.of.stratified.space}. $X$ is the union of finitely many strata. Since $X$ is compact and satisfies \eqref{E:tameness.of.stratified.space}, each stratum is finitely tractable. Thus, $X$ is covered by finitely many tractable neighborhoods 
$(U_{1}, \varphi_{1}), \ldots, (U_{m}, \varphi_{m})$ of various dimensions. Let $i = 1, \ldots, m$. Then $V_{i} := \varphi_{i}(U_{i})$ is convex and bounded. Hence, by \eqref{E:when.Haus.meas.=.Leb.meas},
$\Hm^{p}(V_{i}) < \infty$. (If $\dim U_{i} < p$ then, by \eqref{E:Haus.dim.s-manif.=.s} and definition of ``$\dim$'' in appendix \ref{Chptr:Lip.Haus.meas.dim}, we have $\Hm^{p}(V_{i}) = 0$.) Since $(U_{i}, \varphi_{i})$ is tractable, $\varphi_{i}^{-1}$ is Lipschitz. Therefore, by \eqref{E:Lip.magnification.of.Hm}, 
$\Hm^{p}(U_{i}) = \Hm^{p}\bigl[ \varphi_{i}^{-1}(V_{i}) \bigr]$ is also finite. But, by \eqref{E:Hm.is.mono.and.count.subadditive},
    \begin{equation*}
      \infty = \Hm^{p}(X) \leq \sum_{i=1}^{m} \Hm^{p}(U_{i}) < \infty,
    \end{equation*}
a contradiction that proves part \ref{I:cmpct.strat.space.with.infinite.measure.is.not.tame}.
  \end{proof} 

  \begin{example}[Un-tame stratification] \label{Ex.Untame.strat}
We present two similar examples showing that condition \eqref{E:tameness.of.stratified.space} has some teeth. First, consider the ``topologist's sine curve'' (Munkres \cite[Exercise 1, p.\ 168]{jrM84}), i.e., the space $X \subset \RR^{2}$ that is the union of the closed segment ${0} \times [-1,1]$ and the curve 
$S := \bigl\{ (t, \sin 1/t) : t \in (0,1] \bigr\}$. 
So $X$ is compact. It can be broken down into a finite union of disjoint manifolds as well: 
    \begin{equation*}
      X = \{ (0, -1) \} \cup \bigl( \{0\} \times (-1,1) \bigr) \cup \{ (0,1) \} 
        \cup \{ (t, \sin 1/t) : t \in (0,1) \bigr\} \cup \{ (1, \sin 1) \}.
    \end{equation*}
Thus, $X$ is stratified. (This stratification satisfies the ``condition of the frontier'', Mather \cite[Section 5]{jM2012.MatherStratSpaces}.) This stratification has dimension 1. Therefore, by part \ref{I:dim.of.strat.space.is.same.for.every.stratn} of lemma \ref{L:strat.properties}, any stratification of $X$ has dimension 1. But $X$ is compact with infinite $\Hm^{1}$-measure. Therefore, by part \ref{I:cmpct.strat.space.with.infinite.measure.is.not.tame} of lemma \ref{L:strat.properties}, no stratification of $X$ can satisfy \eqref{E:tameness.of.stratified.space}.

The topologist's sine curve is not locally connected 
(Simmons \cite[p.\ 151]{gfS63}). So consider instead the space
     \begin{equation*}
      X := \bigl\{ (0,0) \bigr\} 
        \cup \Bigl\{ \bigl( t, \sqrt{t} \sin (1/t) \bigr) : 0 < t < 1 \Bigr\} 
          \cup \bigl\{ (1, \sin1 ) \bigr\} .
    \end{equation*}
$X$ is compact and, since 
$\Bigl[ (-t, t) \times \bigl( -  (1+\epsilon) \sqrt{t}, (1+\epsilon) \sqrt{t} \bigr) \Bigr] \cap X$, 
with $\epsilon > 0$ arbitrary, is a neighborhood of $(0,0)$ in $X$ for every $t \in (0,1)$, we see that $X$ is locally connected. 

The length, i.e. $\Hm^{1}$-measure, of $X$ is greater than the length of the arc  
$\bigl\{ (t, \sqrt{t}) : t \in (0,1) \bigr\}$ which is
    \begin{equation*}
     \int_{0}^{1} \sqrt{ 1 + \frac{1}{4 t^{2}} } \, dt 
       = \int_{0}^{1} t^{-1} \sqrt{ t^{2} + \frac{1}{4} } \, dt 
         > \frac{1}{2} \int_{0}^{1} t^{-1} \, dt = \infty .
    \end{equation*}
Hence, by part \ref{I:cmpct.strat.space.with.infinite.measure.is.not.tame} of lemma \ref{L:strat.properties}, $X$ cannot satisfy \eqref{E:tameness.of.stratified.space}.
  \end{example}
  
On the other hand, here are``nice'' stratifications that do satisfy \eqref{E:tameness.of.stratified.space}:

 \begin{example}[Tame stratification]  \label{Ex:tame.stratification}
Recall the definition of a cell, its faces, and a cell complex 
Munkres \cite[pp.\ 71--74]{jrM66}. A cell is a stratified space. The stratification consists of the interiors of the faces of the cell, where the interior of a vertex is the vertex itself. The strata are themselves bounded and convex. Hence, the strata are all finitely tractable. Therefore, by part \ref{I:all.strata.finitely.trctble.means.tame} of lemma \ref{L:strat.properties},
    \begin{multline} \label{E:cell.is.tame.strat.space}
      \text{A cell is a stratified space satisfying 
        \eqref{E:tameness.of.stratified.space}.} \\
      \text{ More generally, a finite cell complex is a stratified space satisfying 
        \eqref{E:tameness.of.stratified.space}.}
    \end{multline}

It follows from part \ref{I:manif.covered.by.tract.nbhds} of lemma \ref{L:strat.properties} that a smooth manifold, regarded as a stratified space with just one stratum, satisfies \eqref{E:tameness.of.stratified.space}. In particular, the assumption in part \ref{I:cmpct.strat.space.with.infinite.measure.is.not.tame} of lemma \ref{L:strat.properties} that $X$ is compact cannot be dropped: $\RR$ is a stratified space (with one stratum) of dimension 1 having infinite $\Hm^{1}$-measure, but also satisfying \eqref{E:tameness.of.stratified.space}.
 \end{example}

Recall \eqref{E:Pf.loc.compct.strat.space} and \eqref{E:D.is.cmpct.Riem.manif}. \emph{Claim:} $\Pf$ is a Borel measurable subset of $\D$. Let $\Rcl_{1}, \ldots, \Rcl_{\ell}$, of $\D$ be the strata of $\Pf$ as a stratified subspace of $\D$ so $\Pf = \bigcup_{i=1}^{\ell} \Rcl_{i}$. Thus, it suffices to show 
    \begin{equation}  \label{E:imbedded.submanif.is.Borel}
      \text{An imbedded submanifold of $\D$ is a Borel measurable subset of $\D$.}
    \end{equation}
Let $\Rcl$ be an imbedded submanifold of $\D$. Being a smooth manifold, $\Rcl$ is second countable (Boothby \cite[Definition (3.1), p.\ 6]{wmB75} or Munkres \cite[Definition 1.1, p.\ 3]{jrM66}). Hence it satisfies Lindel\"of's theorem (Simmons \cite[Theorem A, p.\ 100]{gfS63}). $\Rcl$ is also locally compact (Boothby \cite[Theorem (3.6), p.\ 9]{wmB75}). Therefore, by Ash 
\cite[Theorem A5.15, p.\ 387]{rbA72}, $\Rcl$ can be expressed as the union of countably many compact subsets. A compact subset of $\Rcl$ in the relative topology is compact in $\D$, hence Borel. Thus, $\Rcl$ is a countable union of Borel sets. That means it is Borel. This proves \eqref{E:imbedded.submanif.is.Borel}. The claim that 
$\Pf$ is a Borel, follows. Let $\mcl{Q} \subset \Pf$. By the theorem of A.H. Stone (Milnor and Stasheff \cite[p.\ 66]{jwMjdS74}), $\mcl{Q}$ is paracompact. (See also Hocking and Young \cite[Section 2-11, pp.\ 77--80]{jgHgsY61.Topology}.) Recall that a topological space $\X$ is paracompact if every open cover has an open refinement (i.e., every set in the refinement is a subset of a set in the oringinal cover) that is locally finite. An open cover of $\X$ is locally finite if every point of the space has a neighborhood that intersects only finitely many sets in the cover. So every compact space is paracompact. 

Moreover, $\mcl{Q}$ is second countable because $\D$, as a differentiable manifold (by \eqref{E:D.is.cmpct.Riem.manif}), is, (Boothby \cite[Definition (3.1), p.\ 6]{wmB75} or Munkres \cite[Definition 1.1, p.\ 3]{jrM66}). Hence, by Lindel\"of's theorem (Simmons \cite[Theorem A, p.\ 100]{gfS63}),
    \begin{multline}  \label{E:Pf.has.countable.loc.finite.refinement}
        \text{Any open cover of a subset, } \mcl{Q}, \text{ of } \Pf \\
          \text{ has a countable, locally finite refinement that also covers } \mcl{Q}.
    \end{multline}
E.g., every open cover of $\Pf$ has a countable locally finite refinement 
$\{ \mcl{V}_{i}, i = 1, 2, \ldots \}$ covering $\Pf$: Each $\mcl{V}_{i}$ is a subset of a set in the original cover and every point of the space has a neighborhood that intersects only finitely many sets $\mcl{V}_{i}$.  
\emph{Proof:} Let $\{ \clU_{\alpha}, \alpha \in A \}$ be an open cover of $\Pf$. By paracompactness, there is a locally finite open refinement 
$\{ \mcl{V}_{\beta}, \beta \in B \}$ covering $\Pf$. By Lindel\"of's theorem, 
Simmons \cite[Theorem A, p.\ 100]{gfS63}, 
$\{ \mcl{V}_{\beta} \}$ has a countable subcover.

Let $g \in G$. We assume as usual that the restriction, $g \restriction_{\T}$, of $g$ to $\T$ is a diffeomorphism of $\T$ onto itself. Further, we assume 
	\begin{multline} \label{E:g(Pf)=Pf}
	  g(\Pf) = \Pf. \text{  In fact, if } \Rcl \text{ is a stratum of } \Pf  \\
	    \text{ then } g(\Rcl) \text{ lies entirely inside some other stratum of } \Pf. 
	      \quad (g \in G)
	\end{multline}
Since $G$ is a group, this means that $g(\Rcl)$ is a stratum of $\Pf$ for every $g \in G$. By our definition of stratified space, each $\Rcl$ is an imbedded manifold of $\D$. Since, by \eqref{E:G.is.group.of.diffeos}, each $g: \D \to \D$ is a diffeomorphism, it follows that $g \restriction_{\Rcl}$ is a diffeomorphism of $\Rcl$ onto the stratum $g(\Rcl)$.

Define 
	\begin{equation}  \label{E:R.bar.is.closure.in.P}
		\text{Let } \overline{\Rcl} \text{ denote the closure of } 
		  \Rcl \text{ \emph{in} } \Pf.
	\end{equation} 
I.e., closure in $\Pf$, not $\D$. So $\overline{\Rcl} \subset \Pf$. 

Let $y \in \Pf$. By a ``cone'' at $y$ we will mean a subset $C[y] \subset T_{y} \D$ of dimension $d - p > 0$ ($d := \dim \D$, $p := \dim \Pf$) s.t.\ 
    \begin{equation} \label{E:tw.in.cone}
     \text{ if } w \in C[y] \text{ there exists } \epsilon = \epsilon(w) > 0 \text{ s.t.\ if } 
       t \in [0,1+\epsilon) \text{ then } t w \in C[y].
    \end{equation} 
In particular, the zero vector in $T_{y} \D$ is always in $C[y]$, but $0$ is not the only point in $C[y]$, since $\dim C[y] = d-p > 0$. Simple example: $p = 0$ and $C[y]$ is a ball. Suppose we have chosen a cone $C[y]$ for every $y \in \Pf$. 
Let $\mcl{E} \subset \Pf$ and $I \subset [0, \infty)$. Define 
	\begin{multline}  \label{E:boldF.[E].defns}
	      C[\mcl{E}] := \bigl\{ \bigl( y, v \bigr) \in T \D : (y, v) \in C[y], 
	        \, y \in \mcl{E} \bigr\}
	        = \bigcup_{y \in \mcl{E}} C[y], \\
		  \mbf{F}_{1}[\mcl{E}] := \Bigl\{ \bigl( y, |v|^{-1} v \bigr) \in T \D : 
		    (y,v) \in C[\mcl{E}], v \neq 0 \Bigr\}, \\
		    \text{ and } \mbf{F}_{I}[\mcl{E}] :=  \bigl\{ (y, sv) \in T \D : 
		      \bigl( y, v \bigr) \in \mbf{F}_{1}[\mcl{E}], 
		      \, s \in I \bigr\}.
	\end{multline}
$C[\mcl{E}]$ and $\mbf{F}_{I}[\mcl{E}]$ inherit topology from $T \D$, which, by \eqref{E:D.imbedded.in.Rk} and \eqref{E:Riem.metrics.on.D.on.Rk.same}, inherits topology from $\RR^{2k}$. If $y \in \Pf$, 
write $\mbf{F}_{\centerdot}[y] := \mbf{F}_{\centerdot}\bigl[ \{y\} \bigr]$. 
Here, ``$\centerdot$'' stands either for ``1'' or for an interval ``$I$''. 

We assume that $C[\Pf]$ is $G$-invariant in the sense that for every $g \in G$ 
and $(y,v) \in C[y]$ we have 
$g_{\ast}(y,v) \in C \bigl[ g(y) \bigr]$. A simple argument using \eqref{E:gExp=Expg.star} then shows
	\begin{multline}  \label{E:Gammas.G.invar}
		\text{If } g \in G \text{ and } \mcl{E} \subset \Pf, \text{ then } 
		  g_{\ast} \bigl( \mbf{F}_{1}[\mcl{E}]  \bigr) 
		  = \mbf{F}_{1} \bigl[ g(\mcl{E}) \bigr] \\
		    \text{ and for every } I \subset [0, \infty) \text{ we have } 
		      g_{\ast} \bigl( \mbf{F}_{I}[\mcl{E}]  \bigr) 
		        = \mbf{F}_{I} \bigl[ g(\mcl{E}) \bigr].
	\end{multline}

Here is another way of thinking about cones (see Pflaum \cite[Section 1.1, p.\ 17]{mjP.StratSpaces}). If $\msf{L}$ is a stratified space, define the (open) cone, 
$\msf{CL}$, over $\msf{L}$ to be the quotient space 
$\bigl( [0, 1) \times \msf{L} \bigr)/ \bigl( \{0\} \times \msf{L} \bigr)$. 
Points of $\msf{CL}$ are equivalence classes $\bigl[ (t, x) \bigr]$, 
($0 \leq t < 1, x \in \msf{L}$). The ``cusp'' or ``vertex'' of $\msf{CL}$ is 
the point 
    \begin{equation}  \label{E:cusp}
      \msf{o} := \bigl[ (0,x) \bigr],
    \end{equation} 
where $x \in \msf{L}$ is arbitrary. $\msf{L}$ is the ``link'' of $\msf{CL}$. In order to make $\msf{CL}$ more like $C[y]$ and to put a metric on it we identify $\msf{CL}$ with the following set. Pick $J$ large enough that 
$\msf{L}$ can be imbedded smoothly into $\RR^{J}$. (I.e., the imbedding is smooth on each stratum of $\msf{L}$. In particular, each stratum, $\Rcl$, of $\msf{L}$ is an imbedded submanifold of $\RR^{J}$.) 

Identify $\msf{L}$ with its image in $\RR^{J}$. So the points of $\msf{L}$ are real 
$J$-vectors and $(t,x) \mapsto t(1,x) = (t, tx)$ defines a continuous map 
$[0, 1) \times \msf{L} \to \RR^{J+1}$. 
Then by Munkres \cite[p.\ 112]{jrM84} the map $f : \msf{CL} \to \RR^{J+1}$ defined 
by $f: \bigl[ (t,x) \bigr] \mapsto t(1,x)$ (the scalar $t$ times the vector $(1,x)$) is continuous. It is clearly injective. Identify $\msf{CL}$ with $f(\msf{CL})$:
	\begin{equation}  \label{E:CL.in.Euc.space}
		\msf{L} \subset \RR^{J} \text{ and }
		  \msf{CL} := \bigcup_{0 \leq s < 1} s \cdot \bigl( \{ 1 \} \times \msf{L} \bigr) 
		    \subset \RR^{J+1}.
	\end{equation}
(Here, ``$\cdot$'' indicates scalar multiplication.)  
In particular, 
    \begin{equation} \label{E:vertex=0}
       \text{The vertex, } \msf{o}, \text{ of } \msf{CL} \text{ is identified with the origin, 0.} 
    \end{equation}
Since, as a stratified space, $\msf{L}$ is non-empty, there's more to $\msf{CL}$ than just the vertex. For $s \in [0,1]$ define 
    \begin{equation}  \label{E:scalar.mult.on.cones}
        s \bigl[ (t,x) \bigr] = \bigl[ (st, x) \bigr] \in \msf{CL}. 
    \end{equation}
This definition is compatible with the identification: 
$s f \bigl[ (t,x) \bigr] = f \bigl[ (st,x) \bigr] = f \Bigl( s \bigl[ (t,x) \bigr] \Bigr)$.

Our work will require several metrics. Here is another:
    \begin{equation} \label{E:lambda.metric.defn}
          \text{Put on } \msf{CL} \text{ the metric, } \lambda, 
            \text{ it inherits from } \RR^{J+1}.
    \end{equation} 
Notice that $\lambda$ restricted to $\bigl\{ (1,z)  \in \msf{CL} : z \in \msf{L} \bigr\}$ is just the restriction to $\msf{L}$ of the Euclidean distance on $\RR^{J}$. 
Note that, if $\msf{L}$ is compact, there exists $K < \infty$ s.t.\
	\begin{multline}  \label{E:dist.bnd.on.CL}
		\lambda \Bigl( \bigl[ (s, x) \bigr], \bigl[ (t, y) \bigr] \Bigr) 
		= \bigl| (s, sx) - (t, ty) \bigr|  
		 \leq |s-t| + s|x-y| + |s-t| |y| \\
		   \leq K \bigl( |s-t| + |x-y| \bigr), \quad s,t \in [0,1); x, y \in \msf{L}.
	\end{multline}
Moreover, in general, by \eqref{E:n.c.sqrd.sum.ineq},
    \begin{equation}  \label{E:lambda.revrs.triangle}
        \sqrt{2} \lambda \Bigl( \bigl[ (s, x) \bigr], \bigl[ (t, y) \bigr] \Bigr) 
          \geq |s-t| + |sx - ty|,  \quad s,t \in [0,1); x, y \in \msf{L}.
    \end{equation}
Put on $\Pf \times \msf{CL}$ the metric, 
	\begin{multline}  \label{E:metric.on.Pf.x.CL}
	  (\xi \times \lambda) \Bigl[ \bigl(y, \bigl[(s, z)\bigr] \bigr) , 
	    \bigl(y', \bigl[(s', z')\bigr] \bigr) \Bigr]
	  := \xi(y, y') + \lambda \Bigl( \bigl[ (s, z) \bigr], \bigl[ (s', z') \bigr] \Bigr), \\   
		\quad y, y' \in \Pf, \, [s, z], [s', z'] \in \msf{CL},
	\end{multline}
where $\xi$ is the topological metric on $\D$ (see \eqref{E:xi.is.metric.on.D}) 
and $\lambda$ is the metric on $\msf{CL}$. 
In light of all this, we sometimes write 
$\bigl(y, (s, sz) \bigr)$ or even $(y, s, sz)$ instead 
of $\bigl(y, \bigl[(s, z)\bigr] \bigr)$. (Here, $y \in \Pf$, $s \in [0,1)$, 
$z \in \msf{L}$.)

In part \ref{I:local.triv} in the following we connect the two ways of viewing cones. Recall (see lemma \ref{L:scalar.mult.xi+.cont}) the definition of the restriction of a fiber bundle, say $T \D$, to a subspace, say $\Pf$. 
    \begin{definition} \label{D:fibering.by.cones}
Let $\Pf \subset \D$ be a stratified space of dimension $p < d := \dim \D$. Let $G$ be a finite group of diffeomorphisms on $\D$. Suppose \eqref{E:g.preserves.Riem.met} holds  and $\Pf$ is $G$-invariant: $g(\Pf) = \Pf$ for every $g \in G$. Say that an (open) neighborhood of $\Pf$ in the restriction $T \D \restriction_{\Pf}$
is ``fibered over $\Pf$ with, as fibers, open cones $C[y] \subset T_{y} \D$ 
($y\in \Pf$)" if the following holds. Let 
	\begin{equation*}
		C[\Pf]  = \bigcup_{y \in \Pf} C[y] . 
	\end{equation*} 
Put on $C[\Pf]$ the topology it inherits from $T \D$. 
Recall the definition \eqref{E:pi.is.proj} of $\pi : T \D \to \D$. 
Let $\pi_{C} := \pi \restriction_{C[\Pf]} : C[\Pf] \to \Pf$, the restriction of $\pi$ to $C[\Pf]$, be the bundle projection 
    \begin{equation*}
      \pi_{C}(y,w) :=  y \text{ for } (y, w) \in C[\Pf] . 
    \end{equation*} 
  \begin{enumerate}
	\item (Local Triviality) \label{I:local.triv}
	$C[\Pf]$ is ``locally trivial'' in the following sense. Let $x \in \Pf$. 
	Then $x$ has an open neighborhood, $\mcl{V}$, in $\Pf$ that can be written 
	$\mcl{V} = \bigcup_{i=1}^{n} \mcl{A}_{\mcl{V},i}$ for some 
	$n = n_{\mcl{V}} = 1, 2, \ldots$, where each $\mcl{A}_{\mcl{V},i}$ is \emph{closed} 
	in $\mcl{V}$ but $\mcl{A}_{\mcl{V},i} \cap \Rcl$ 	is \emph{open} in $\Rcl$ 
	for every $i$ and every stratum $\Rcl$ of $\Pf$. (The $\mcl{A}_{\mcl{V},i}$'s might not be disjoint.) In the following let $\mcl{V}$ and $i = 1, \ldots, n_{\mcl{V}}$ be arbitrary but fixed and drop 
$\mcl{V}$ from subscripts.
   \begin{enumerate} 
        \item There is a compact stratified space $\msf{L}_{i}$ (a link) 
of dimension $d-p-1$ (so $\dim \msf{CL}_{i} = d-p$) and a continuous \emph{injection}, 
$h_{i} : \mcl{A}_{i} \times \msf{CL}_{i} \to \pi_{C}^{-1}(\mcl{A}_{i}) 
= C[\mcl{A}_{i}] \subset C[\Pf]$ mapping 
$\mcl{A}_{i} \times \msf{CL}_{i}$ homomorphically onto its image. But $h_{i}$ does not have to be surjective.
  \label{I:L.d-p-1}
  \item The collection $(\mcl{A}_{j}, \msf{L}_{j}, h_{j})$, $j = 1, \ldots, n)$ satisfies the following.
    \begin{multline}  \label{E:compatibility.across.As}
      \text{If } h_{ \ell} \bigl( x, s \, (1,z) \bigr) = h_{ m} \bigl( x, s' \, (1,z') \bigr)
        \text{ then } s = s', \\
          x \in \mcl{A}_{ \ell} \cap \mcl{A}_{ m}; \; s, s' \in [0,1) ; 
            \; z \in  \msf{L}_{\ell}, \; z' \in  \msf{L}_{m} ; \; 
              \ell, m = 1, \ldots, n .
    \end{multline}
However, it is \emph{not} required that $\msf{L}_{\ell},$ and $\msf{L}_{m}$ be homeomorphic. \label{I:L's.are.compatible}
	\item $\mcl{A}_{i}$ is a stratified space with strata $\mcl{A}_{i} \cap \Rcl$ for all strata 
	$\Rcl$ of $\Pf$ for which the intersection is non-empty. Moreover, 
	$\mcl{A}_{i}$ satisfies \eqref{E:tameness.of.stratified.space}. $\msf{L}_{i}$ also satisfies \eqref{E:tameness.of.stratified.space}. Put on $\mcl{A}_{i} \times \msf{CL}_{i}$ the restriction to $\mcl{A}_{i} \times \msf{CL}_{i}$ of the metric $\xi \times \lambda$ defined in \eqref{E:metric.on.Pf.x.CL} with $\msf{CL}_{i}$ in place of $\msf{CL}$. Denote that restriction by $\xi \times \lambda_{i} = \xi \times \lambda_{i}$. \label{I:L.A.tameness}
	\item We have 
		\begin{equation}   \label{E:pi.-1.V.=.union.of.As}
			\bigcup_{j=1}^{n} h_{j}(\mcl{A}_{j} 
			  \times \msf{CL}_{j}) = \pi_{C}^{-1}( \mcl{V} ) = C[\mcl{V}].
		\end{equation}
		\label{I:CV.from.h.Ai}
	\item $h_{i}$ and its inverse 
	$h_{i}^{-1} : h_ {i}(\mcl{A}_{i} \times \msf{CL}_{i}) \to \mcl{A}_{i} 
	\times \msf{CL}_{i}$ are Lipschitz. (Use the metrics $\xi_{+}$ defined in \eqref{E:xi+.from.2.metrics} 
	on $h_ {i}(\mcl{A}_{i} \times \msf{CL}_{i}) \subset C[\Pf] 
	\subset \T \D \restriction_{\Pf}$ and $\xi \times \lambda_{i}$ 
	on $\mcl{A}_{i} \times \msf{CL}_{i}$.) Thus, $h_{i}$ is bi-Lipschitz 
	(\eqref{E:bi-Lipschitz.defn}). \label{I:h.hi.invrs.Lip}
	\item $\pi_{C} \circ h_{i}(y, w) = y = (y,0)$, whenever 
	  $(y, w) \in  \mcl{A}_{i} \times \msf{CL}_{i}$.   \label{I:pi.h.(y,w)=y}
	\item $h_{i}$ is homogeneous:
	   \begin{multline} \label{E:homogeneity.of.hi}
		h_{i} \Bigl( y, \bigl[ (s t, z) \bigr] \Bigr) 
		  = h_{i} \Bigl( y, s \bigl[ (t, z) \bigr] \Bigr)
		    = s \, h_{i} \Bigl( y, \bigl[ (t, z) \bigr] \Bigr) \\
		      y \in \Pf, s \in [0,1],  t \in [0,1), z \in \msf{L}_{i}.
	   \end{multline}
(See \eqref{E:vector.ops.on.TD} and \eqref{E:scalar.mult.on.cones}.) In particular, $h_{i}(y, \msf{o}) = y$ ($y$ identified with $(y,0)$ as usual), 
where $\msf{o}$ is the vertex of $\msf{CL}_{i}$. (See \eqref{E:vertex=0} and \eqref{E:cusp}.)
	  \label{I:homogeneity.of.hi}
   \end{enumerate}  \label{I:local.triv}
   \item $\mcl{C} := Exp \bigl( C[\Pf] \bigr)$ is an open neighborhood of $\Pf$ in $\D$
   and $Exp \restriction_{C[\Pf]} : C[\Pf] \to \mcl{C}$ is a homeomorphism. Thus, $C[\Pf]$ is an open neighborhood of $\Pf$ in $T \D \restriction_{\Pf}$. Let $\alpha := 
   \bigl( Exp \restriction_{C[\Pf]} \bigr)^{-1}: \mcl{C} \to C[\Pf]$.
   Thus, $(y,w) = \alpha \bigl[ Exp_{y}(w) \bigr]$ ($(y,w) \in C[\Pf]$). 
   In particular, $\alpha(y) = (y,0)$ ($y \in \Pf$). 
   If $\mcl{K} \subset \Pf$ is compact, then $Exp$ is Lipschitz on $C[\mcl{K}]$ (see \eqref{E:boldF.[E].defns}) and $\alpha$ is Lipschitz on 
   $Exp \bigl( C[\mcl{K}] \bigr) \subset \mcl{C}$ w.r.t.\ $\xi_{+}$ (equation \eqref{E:xi+.from.2.metrics}) and $\xi$ (see \eqref{E:xi.is.metric.on.D}). 
   \label{I:Exp.alpha.homeom}  
    \item $\mcl{C}$ and $C[\Pf]$ are $G$-invariant: 
	    \begin{equation}  \label{E:C.C[P].G.invar}
		g(\mcl{C}) = \mcl{C} \text{ and } g_{\ast} \bigl( C [y] \bigr)  
		  = C \bigl[ g(y) \bigr] 
		    \text{ for every } y \in \Pf \text{ and } g \in G.
	    \end{equation} 
     \label{I:C.C[P].G.invar}
  \end{enumerate}  
    \end{definition}

Compatibility of $h$'s for overlapping $\mcl{A}$'s, beyond that posited in \eqref{E:compatibility.across.As}, is established in remark \ref{R:local.compatibility}. In remark \ref{R:extending.hi}, it shown that the maps $h_{i}$ can be extended while retaining the same properties. 

In appendix \ref{Chptr:rob.loc.circle.cones.appendix} we construct a cone bundle as in the definition over a space, though not a manifold, is still rather simple. I found the task annoying difficult. It wasn't the construction so much as verifying it had the right properties. It would be nice to have some theoretical tools to make such work easier. Those tools may already exist, at least partly. Pflaum \cite[Corollary 3.9.3, p.\ 143]{mjP.StratSpaces} shows that ``every Whitney stratified space is locally trivial with cones as typical fibers.'' Section 3.10 ibid discusses ``cone spaces''. See also Banagl \cite[Proposition 6.2.5, p.\ 130]{mB07.StratSpaces}. 
I found these works helpful as I was formulating the preceding definition.

For the remainder of this chapter we assume
    \begin{equation} \label{E:there.are.cone.fibers.over.Pf} 
      \Pf \text{ has a neighborhood in } T \D \restriction_{\Pf} 
      \text{ fibered over } \Pf \text{ with open cone fibers }. 
    \end{equation}  
 
Note that in part \ref{I:L.d-p-1}, $\dim \msf{CL}_{i}$ is constant, $d-p$, in $i$. 
That, together with part \ref{I:Exp.alpha.homeom}, viz., ``$\mcl{C} := Exp \bigl( C[\Pf] \bigr)$ is an (open) neighborhood of $\Pf$'', constrains the stratification of $\Pf$. 
In general, the cones $Exp(C[y])$, $y \in \Pf$, does not constitute a foliation 
(Lawson \cite{hbL74.Foliations}) of the neighborhood $Exp(\clU)$ because the fibers $Exp(C[y])$ might not be manifolds (example \ref{Ex:chevron.cross}). 

  \begin{example}[Dimension 0] \label{E:cones.over.0.dim.Pf}
Suppose $\dim \Pf = 0$. Hence, by \eqref{E:0.dim.stratified.spaces}, we know that $\Pf$ must be finite, say $\Pf = \{ y_{1}, \ldots, y_{m} \}$. 
Let $\delta > 0$ satisfy $2 \delta < \min \bigl\{ |y_{i} - y_{j}| : i \neq j \bigl\}$ In part \ref{I:local.triv} of the definition take 
the $\mcl{V}$'s to be the individual points $y_{i}$. Let $i = 1, \ldots, m$, let 
$n = 1$, and let $\mcl{A}_{1} = \mcl{V} = \{ y_{i} \}$. 
Take $\msf{L}_{1} := S^{d-1}$, the unit $(d-1)$-sphere. (By lemma \ref{L:strat.properties} part \ref{I:cmpct.manif.finitely.tractble}, it satisfies \eqref{E:tameness.of.stratified.space}.) Define 
$h_{1} : \{ y_{i} \} \times \msf{CL}_{i} \to C[y_{i}]$ 
by $h_{1}\bigl( y_{i}, (t, x) \bigr) = \bigl( y_{i}, t \delta x \bigr)$ 
($t \in [0,1)$, $x \in \msf{L}_{1} = S^{d-1}$). Thus, $C[y_{i}]$ is just the open ball 
in $T_{y_{i}} \D$ with center 0 and radius $\delta$. 
  \end{example}

Let $\mcl{A}_{i}$, $\msf{L}_{i}$, and $h_{i} : \mcl{A}_{i} \times \msf{L}_{i} \to C[\mcl{A}_{i}]$ be as in part \ref{I:local.triv} of definition \ref{D:fibering.by.cones}. Recall \eqref{E:Riem.metrics.on.D.on.Rk.same}. \emph{Claim:} 
    \begin{equation} \label{E:hi.is.bounded}
        \bigl| h_{i} \bigr| \text{ is bounded on } \mcl{A}_{i} \times \msf{CL}_{i}.
    \end{equation} 
(See \eqref{E:Riem.metrics.on.D.on.Rk.same}.) For let $y \in \mcl{A}_{i}$, 
$z \in \msf{L}_{i}$, and $s \in [0,1)$. Since $\msf{L}_{i}$ is compact there exists 
$K' < \infty$ s.t.\ $\bigl| (1,w) \bigr| \leq K'$ for every $w \in \msf{L}_{i}$. Recall the metric definitions \eqref{E:xi+.from.2.metrics}, \eqref{E:metric.on.Pf.x.CL}, and \eqref{E:metric.on.Pf.x.CL}. Then by part \ref{I:h.hi.invrs.Lip} 
of the definition there exists $K < \infty$ s.t., by \eqref{E:vector.ops.on.TD},
compactness of $\msf{L}_{i}$, and \eqref{E:homogeneity.of.hi},
    \begin{align*}
      \Bigl| h_{i} \bigl( y, s(1,z) \bigr) \Bigr| 
          &= \xi_{+} \Bigl[ h_{i} \bigl( y, s(1,z) \bigr), h_{i} (y, \msf{o}) \Bigr] \\ 
          &\leq K (\xi \times \lambda_{i}) \Bigr[ \bigl( y, s(1,z) \bigr), (y, \msf{o}) \Bigr] \\
          &= K \bigl| s(1,z) \bigr| \leq s K K' \leq K K' < \infty.
    \end{align*}
This proves the claim.

Let $\mcl{K} \subset \Pf$ be compact. We have the following corollary of \eqref{E:hi.is.bounded}.
    \begin{equation} \label{E:C[K].rel.compact}
        \text{There exists } M < \infty \text{ s.t.\ if } (y,v) \in C[\mcl{K}] \text{ then } |v| < M. 
    \end{equation}
(This is strengthened in \eqref{E:|(y,v)|.bounded}.) Thus, $C[\mcl{K}]$ is relatively compact in $T \D \restriction_{\Pf}$. To prove \eqref{E:C[K].rel.compact}, first observe that by compactness, there are finitely many sets $\mcl{V}$ as in part \ref{I:local.triv} of definition \ref{D:fibering.by.cones} 
that cover $\mcl{K}$. That means there are finitely many sets $\mcl{A}_{i}$ as in part \ref{I:local.triv} of definition that cover $\mcl{K}$. By \eqref{E:pi.-1.V.=.union.of.As}, $C[\mcl{K}]$ lies in the union of the images of the form $h_{i} \bigl[ \mcl{A}_{i} \times \msf{CL}_{i} \bigr]$. But, by \eqref{E:hi.is.bounded} and \eqref{E:homogeneity.of.hi} again, the vectors in $h_ {i}(\mcl{A}_{i} \times \msf{CL}_{i})$ are bounded in length. The claim follows.

Actually, the requirement in definition \ref{D:fibering.by.cones}\eqref{I:Exp.alpha.homeom} that $Exp$ be Lipschitz on $C[\mcl{K}]$ for any $\mcl{K} \subset \Pf$ compact is unnecessary. This is a consequence of lemma \ref{L:Exp.loc.Lip}, \eqref{E:local.Lip.is.Lip.on.compacts}, and \eqref{E:C[K].rel.compact}.

By definition \ref{D:fibering.by.cones}(\ref{I:L.d-p-1}) $h_{i}$ is an injection. By part \eqref{I:homogeneity.of.hi}, we have $\bigl| h_{i}(y, \msf{o}) \bigr| = 0$. Therefore,
    \begin{equation}  \label{E:|h|=0.condition}
      \text{For } y \in \Pf, t \in [0,1), z \in \msf{L}_{i} \text{ we have }
        \left| h_{i} \Bigl( y, \bigl[ (t, z) \bigr] \Bigr) \right| = 0 \text{ if and only if } t = 0 . 
    \end{equation}
    
A fact in the opposite direction to \eqref{E:C[K].rel.compact} is the following. Let $\mcl{V}$, $n$, $\mcl{A}_{i}$, $\msf{L}_{i}$, and $h_{i}$ ($i = 1, \ldots,n$) be as in part \ref{I:local.triv} of definition \ref{D:fibering.by.cones}. Thus, $h_{i} : \mcl{A}_{i} \times \msf{CL}_{i} \to C[\mcl{A}_{i}]$. Recall \eqref{E:vector.ops.on.TD}. Then
    \begin{equation} \label{E:unif.lwr.bound.on(y,v)|>0.for.y.in.Ai}
        \text{There exists } b_{i} > 0 \text{ s.t.\ for every }  
          (y, z) \in \mcl{A}_{i} \times \msf{L}_{i} 
          \text{ we have } \left| h_{i} \left( y, \tfrac{1}{2}, \tfrac{1}{2} z \right) \right| > b_{i}
    \end{equation}
In particular, for every $y \in \mcl{A}_{i}$ there exists $v \in \RR^{k}$ with 
$(y, v) \in C[y]$ s.t.\  $|v| > b_{i}$. To see this suppose \eqref{E:unif.lwr.bound.on(y,v)|>0.for.y.in.Ai} is false. Then there exists a sequence 
$\bigl\{ (y_{m}, z_{m} ) \bigr\} \subset \mcl{A}_{i} \times \msf{L}_{i}$ s.t.\ 
$\left| h_{i} \left( y_{m}, \tfrac{1}{2}, \tfrac{1}{2} z_{m} \right) \right| \to 0$. 
Let $(y_{m}, v_{m}) := h_{i} \left( y_{m}, \tfrac{1}{2}, \tfrac{1}{2} z_{m} \right) 
\in C[y_{m}]$. 
Thus, $v_{m} \to 0$. Recall that, by part \ref{I:h.hi.invrs.Lip} of definition \ref{D:fibering.by.cones}, $h_{i}^{-1}$ is Lipschitz. Therefore, there exists $K < \infty$ s.t.\
    \begin{equation*}
        \tfrac{1}{2} \leq \left| \tfrac{1}{2}(1, z_{m}) \right| 
          = (\xi \times \lambda_{i}) \bigl[ h_{i}^{-1}(y_{m}, v_{m}), 
            h_{i}^{-1}(y_{m}, 0) \bigr] \leq K \xi_{+} \bigl[ (y_{m}, v_{m}), 
              (y_{m}, 0) \bigr] = K |v_{m}| \to 0.
    \end{equation*}
Contradiction.

Observe that \eqref{E:homogeneity.of.hi} implies that 
		\begin{align} \label{E:homogeneity.of.hi.plus}
			h_{i} \Bigl( y, \bigl[ (s t, z) \bigr] \Bigr) 
			  = s \, h_{i} \Bigl( y, \bigl[ (t, z) \bigr] \Bigr), 
			    \quad &y \in \mcl{A}_{i}, z \in \msf{L}_{i}, \text{ and } \\
		&s \in  
                    \begin{cases}
                      [0, \infty), &\text{ if } t = 0 , \\
                      [0,t^{-1}), &\text{ if } t \in (0,1) .
                    \end{cases} \notag
		\end{align}
To see this, let $y \in  \mcl{A}_{i}$ $z \in \msf{L}_{i}$. If $t = 0$ and 
$s \in [0,\infty)$ then by \eqref{E:vector.ops.on.TD}, 
$h_{i} \Bigl( y, \bigl[ (s t, z) \bigr] \Bigr) = s \, h_{i} \Bigl( y, \bigl[ (t, z) \bigr] \Bigr)$ is equivalent to $0 = 0$.  
So assume $t \in (0,1)$ and $s \in [0,t^{-1})$. If $s \in [0,1]$, then
$h_{i} \Bigl( y, \bigl[ (s t, z) \bigr] \Bigr) = s \, h_{i} \Bigl( y, \bigl[ (t, z) \bigr] \Bigr)$
is just \eqref{E:homogeneity.of.hi}. So assume $s \in (1,t^{-1})$.
Then $s^{-1} \in (t, 1)$. Hence, by \eqref{E:homogeneity.of.hi}, 
	\begin{equation*}
		s h_{i} \Bigl( y, \bigl[ (t, z) \bigr] \Bigr) 
		  = s h_{i} \Bigl( y, \bigl[ (s s^{-1} t, z) \bigr] \Bigr) 
		    = s s^{-1} \, h_{i} \Bigl( y, \bigl[ (s t, z) \bigr] \Bigr).
	\end{equation*}
(See remark \ref{R:extending.hi}.)
    	
 \begin{example} \label{Ex:chevron.cross}
Figure \ref{F:ConeBundle} shows an example of a neighborhood of a stratified space fibered over $\Pf$ by cones. In that example, the union of the crossed black lines represent $\Pf$. The point where the lines intersect represent $\T$. 
Interpret ``left part", ``right part", ``top part", and ``bottom part" of $\Pf \setminus \T$ in the obvious way as open line segments that do not include $\T$. Stratify $\Pf$ into $\T$ plus the left, right, top, and bottom parts of $\Pf \setminus \T$. The red shapes are a sample of the fibers. Except at the origin the fibers are chevrons or line segments. At the origin the fiber is an ``X''. Thus, at the origin the fiber is not a manifold. It also illustrates the point that the \emph{fibers over $\Pf$ do not have to be homeomorphic.}

Let $x \in \Pf$. Suppose $x \neq \T$ and, for concreteness, suppose $x$ belongs to the right part of $\Pf \setminus \T$. Then take $\mcl{V}$ to be the right part. Let $n = 1$, 
$\mcl{A}_{1} := \mcl{V} =$ open right part, and $\msf{L}_{1} = $ two point space. 
The homeomorphism 
$h_{1} : \mcl{A}_{1} \times \msf{CL}_{1} \to \pi_{C}^{-1}(\mcl{A}_{1})$ is obvious and obviously bi-Lipschitz.

Now suppose $x = \T$. Take $\mcl{V} := \Pf$ and $n = 4$. Take $\mcl{A}_{1}$ to be the ``closed right part'' of $\Pf$, i.e., $\text{(right part)} \cup \T$. Let $\mcl{A}_{2}$ to be the ``closed top part'' of $\Pf$. Define $\mcl{A}_{3}, \mcl{A}_{4}$ in the obvious similar way. 
(So $\mcl{A}_{i}$'s are not disjoint.) For $i = 1, \ldots, 4$ let $\msf{L}_{i}$ be the two-point space as before. Note that for $i = 1, \ldots, 4$ we have 
that $h_{i} \bigl( \T \times \msf{CL}_{i} \bigr)$ is only a proper subset of the cross, 
$C[\T]$. For example, $h_{1} \bigl( \T \times \msf{CL}_{1} \bigr)$ is a chevron pointing to the left. Now, $\mcl{A}_{i}, \msf{L}_{i}$ ($i = 1, \ldots, 4$) suffice, but one might 
also let $n = 5$, $\mcl{A}_{5} := \T$, and let $\msf{L}_{5}$ be the four-point space. Thus, the links $\msf{L}_{i}$ do not have to be homeomorphic.
 \end{example}
 
  \begin{figure}
      \epsfig{file = 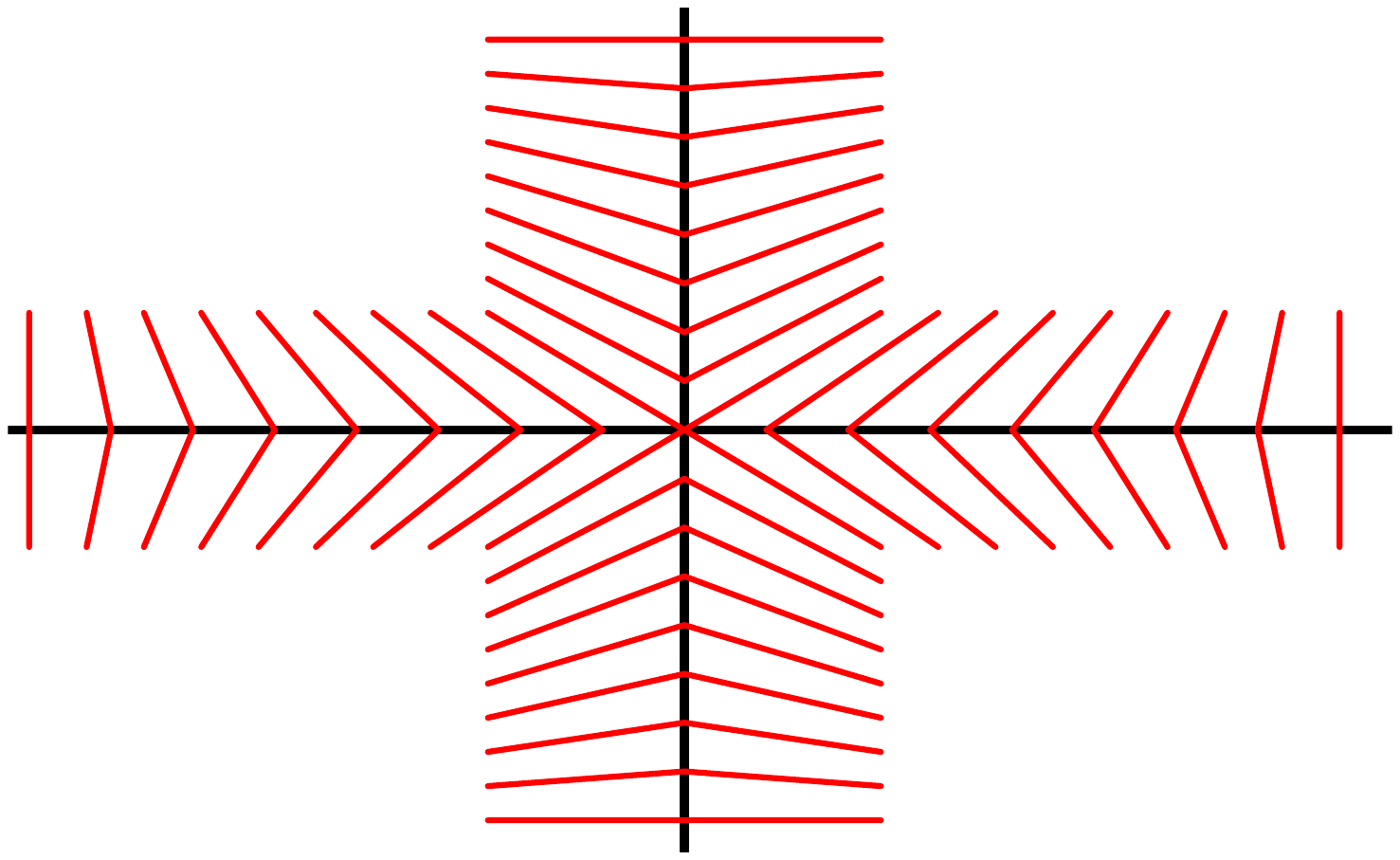, height = 5.8in, , }  
      \caption{Black lines form a stratified space sitting in the plane. Red chevrons and cross are cones, the full collection of which fiber a neighborhood of the stratified space. 
(See example \ref{Ex:chevron.cross}). }    
             \label{F:ConeBundle}
  \end{figure}

  \begin{remark}[Compatibility of overlapping $\mcl{A}_{i}$'s]  \label{R:local.compatibility}
\eqref{E:compatibility.across.As} (extended in \eqref{E:extended.compatibility.across.As} below) establishes some compatibility on overlaps of $\mcl{A}_{i}$'s. Here we develop that idea further. Suppose $\mcl{V}$, $\mcl{V}'$ are neighborhoods as in the definition. Suppose $(\mcl{A}_{i}, \msf{L}_{i}, h_{i})$ 
and $(\mcl{A}'_{j}, \msf{L}'_{j}, h'_{j})$  ``belong'' to  $\mcl{V}$, $\mcl{V}'$, resp.\ and suppose
$X := h_{i}(\mcl{A}_{i} \times \msf{CL}_{i}) \cap h'_{j}(\mcl{A}'_{j} \times \msf{CL}'_{j}) 
\neq \varnothing$, 
so $X \subset C[\mcl{A}_{i} \cap \mcl{A}'_{j}]$. Similar to \eqref{E:vector.ops.on.TD}, for $\bigl( y, (1,z) \bigr) \in \mcl{A}_{i} \times \msf{CL}_{i}$ and $s \in [0,1]$, define 
    \begin{equation*}
      s \bigl( y, (1,z) \bigr) = \bigl( y, s(1,z) \bigr).
    \end{equation*} 
Similarly for $\mcl{A}'_{i} \times \msf{CL}'_{i}$. 
Define $f := h_{i}^{-1} \circ h'_{j} \; : \; (h'_{j})^{-1}(X) \to h_{i}^{-1}(X)$. Let $(y, v) \in X$. 
Then, there exists $t, t' \in \RR$, $z \in \msf{CL}_{i}$, and $z' \in \msf{CL}_{j}$ s.t.\ 
$h_{i} \bigl[ \bigl( y, t(1,z) \bigr) \bigr] = (y,v) = h'_{j} \bigl[ \bigl( y', t'(1,z') \bigr) \bigr]$. 
Let $s \in [0,1]$. By \eqref{E:homogeneity.of.hi}, 
       \begin{multline*}
          s f \bigl[ \bigl( y', t'(1,z') \bigr) \bigr]  
            = s (h_{i}^{-1} \circ h'_{j}) \bigl[ \bigl( y', t'(1,z') \bigr) \bigr] 
                = s h_{i}^{-1} \bigl[ (y,v) \bigr] = s \bigl( y, t(1,z) \bigr) \\
                  = \bigl( y, st(1,z) \bigr) = h_{i}^{-1} (y, sv) 
                  = h_{i}^{-1} \circ h_{j} \bigl[ \bigl( y', st'(1,z') \bigr) \bigr] \\
                    = f \bigl[ \bigl( y', st'(1,z') \bigr) \bigr]
                    = f \bigl[ s \bigl( y', t'(1,z') \bigr) \bigr] .
       \end{multline*}
So overlap implies a compatibility between $h_{i}$ and $h'_{j}$.
  \end{remark}
 
 \begin{remark}  \label{R:G.acts.on.V.etc}
The group $G$ operates on the machinery in part \ref{I:local.triv} of the definition. Specifically, let $g \in G$. 
Then, if $x \in \Pf$ and $\mcl{V}, n, \mcl{A}_{i}, \msf{L}_{i}, h_{i}$ ($i = 1, \ldots, n$) are as in the definition, 
with $x \in \mcl{V}$, then $g(\mcl{V})$ is a neighborhood of $g(x)$. 
Now, $g$ is a homeomorphism that, by \eqref{E:g(Pf)=Pf}, maps strata onto strata. 

I am embarrassed to admit it, but I need to note the following for repeated use. Let $X$ and $Y$ be nonempty sets and let $g : X \to Y$ be a bijection. 
Let $\{ \mcl{A}_{\alpha}, \alpha \in I \}$  be a family of subsets of $X$. 
	\begin{equation}  \label{E:g.commutes.w/.set.ops}
		g \left( \bigcup_{\alpha} \mcl{A}_{\alpha} \right) 
		  = \bigcup_{\alpha} g(\mcl{A}_{\alpha}), \quad
		    g \left( \bigcap_{\alpha} \mcl{A}_{\alpha} \right) 
		      = \bigcap_{\alpha} g(\mcl{A}_{\alpha}), 
		         \text{ and }g( \mcl{A}_{\alpha}^{c} ) = g(\mcl{A}_{\alpha})^{c},
	\end{equation}
where ``${}^{c}$'' indicates set complementation w.r.t.\ $X$ or $Y$, whichever is appropriate. Note that the first equality holds if $g$ is an arbitrary map.

Therefore, by \eqref{E:g.commutes.w/.set.ops}, 
$g(\mcl{V}) = \bigcup_{i=1}^{n} g(\mcl{A}_{i})$ and, for every $i$, 
$g(\mcl{A}_{i})$ is \emph{closed} in $g(\mcl{V})$ but \emph{open} in every stratum $\Rcl$. 

Let $g \in G$ and let $\mcl{Q} \subset \Pf$. Then, by \eqref{E:C.C[P].G.invar}, 
$g_{\ast} \bigl( C[ \mcl{Q} ] \bigr) \subset C \bigl[ g(\mcl{Q}) \bigr]$. 
Replacing $g$ by $g^{-1}$ and $\mcl{Q}$ by $g(\mcl{Q})$ we get 
$g_{\ast}^{-1} \bigl( C[ g(\mcl{Q}) ] \bigr) \subset C [\mcl{Q}]$. Combining the two yields
    \begin{equation} \label{E:g.star.C[U]=C[g(U)]}
      g_{\ast} \bigl( C[\mcl{Q}] \bigr) = C \bigl[ g(\mcl{Q}) \bigr] ,
        \qquad g \in G, \; \mcl{Q} \subset \Pf .
    \end{equation}

Let $i = 1, \ldots, n$; $y \in g(\mcl{A}_{i})$; $t \in [0, 1)$; and $z \in \msf{L}_{i}$. Thus, 
$\Bigl( g^{-1}(y), \bigl[ (t, z) \bigr] \Bigr) \in \mcl{A}_{i} \times \msf{CL}_{i}$. 
Let $h :  \mcl{A}_{i} \times \msf{CL}_{i} \to C [ \mcl{A}_{i} ]$ 
satisfy $\pi \circ h \Bigl( x, \bigl[ (t, z) \bigr] \Bigr) = x$ for $x \in \mcl{A}_{i}$. E.g., by definition \ref{D:fibering.by.cones} part \ref{I:pi.h.(y,w)=y}, $h = h_{i}$ has this property.
Hence, 
    \begin{equation*}
      h \Bigl( g^{-1}(y), \bigl[ (t, z) \bigr] \Bigr) \in C\bigl[ g^{-1}(y) \bigr] 
        \subset C [ \mcl{A}_{i} ] .
    \end{equation*}
Therefore, by \eqref{E:C.C[P].G.invar}, we have 
$g_{\ast} \left[ h \Bigl( g^{-1}(y), \bigl[ (t, z) \bigr] \Bigr) \right] \in C[y] 
\subset C \bigl[ g(\mcl{A}_{i}) \bigr]$. 
Define 
    \begin{multline} \label{E:grp.action.on.h}
      gh \Bigl( y, \bigl[ (t, z) \bigr] \Bigr) 
        := g_{\ast} \left[ h \Bigl( g^{-1}(y), \bigl[ (t, z) \bigr] \Bigr) \right] \in C[y] 
          \subset C \bigl[ g(\mcl{A}_{i}) \bigr], \\
            y \in g(\mcl{A}_{i}), \; t \in [0, 1), \text{ and } z \in \msf{L}_{i}.
    \end{multline}
Thus, $gh : g(\mcl{A}_{i}) \times \msf{CL}_{i} \to C \bigl[ g(\mcl{A}_{i}) \bigr]$. Conversely, 
let $(y,v) \in C \bigl[ g(\mcl{A}_{i}) \bigr]$. By \eqref{E:g.star.C[U]=C[g(U)]}, there exists $(x,w) \in C[\mcl{A}_{i}]$ s.t.\ $(y,v) = g_{\ast}(x,w)$. There exist $t \in [0, 1)$ 
and $z \in \msf{L}_{i}$ s.t.\ 
$(x,w) = h_{i} \Bigl( x, \bigl[ (t,z) \bigr] \Bigr) = h_{i} \Bigl( g^{-1}(y) \bigl[ (t,z) \bigr] \Bigr)$. Hence,
    \begin{equation*}
       g h_{i} \Bigl( y, \bigl[ (t,z) \bigr] \Bigr) 
         = g_{\ast} \left[ h_{i} \Bigl( g^{-1}(y) \bigl[ (t,z) \bigr] \Bigr) \right] = g_{\ast}(x,w)
           = (y,v) .
    \end{equation*}
 Thus, $gh : g(\mcl{A}_{i}) \times \msf{CL}_{i} \to C \bigl[ g(\mcl{A}_{i}) \bigr]$ is a bijection.

This is a group action. By \eqref{E:grp.action.on.h}, 
$\pi \circ gh \Bigl( y, \bigl[ (t, z) \bigr] \Bigr) = y$,  so $gh$ has the properties of an ``$h$''. 
Let $g_{1}, g_{2} \in G$. Let $y \in (g_{2} \circ g_{1})(\mcl{A}_{i})$, $t \in [0,1]$, 
and $z \in \msf{L}_{i}$. Then 
	\begin{align*}
		g_{2}( g_{1} h ) \Bigl( y, \bigl[ (t, z) \bigr] \Bigr) 
		  &= g_{2\ast} \left[ (g_{1} h) \Bigl( g_{2}^{-1}(y), \bigl[ (t, z) \bigr] \Bigr) \right] \\ 
		  &= g_{2\ast} \left[ g_{1\ast} \Bigl( h \Bigl( g_{1}^{-1} \circ g_{2}^{-1}(y), 
		    \bigl[ (t, z) \bigr] \Bigr) \Bigr) \right] \\
		  &= (g_{2} \circ g_{1})_{\ast} \left[ h \Bigl( (g_{2} \circ g_{1})^{-1}(y), 
		     \bigl[ (t, z) \bigr] \Bigr) \right] \\
		  &= (g_{2} \circ g_{1}) h \Bigl( y, \bigl[ (t, z) \bigr] \Bigr).
	\end{align*}
We show that $gh_{i}$ has the properties of an $h_{i}$. $g$ is a diffeomorphism, so injective. Therefore, by functoriality, $g_{\ast}$ is injective. Since, by part \ref{I:L.d-p-1} of definition \ref{D:fibering.by.cones}, $h_{i}$ is injective, so is $gh_{i}$. Moreover, by Boothby \cite[Exercise 6, p.\ 337]{wmB75}, $g_{\ast} : T \D \to T \D$ is a diffeomorphism, hence, a homeomorphism. Since $h_{i}$ is a homeomorphism onto its image, so is $gh_{i}$. Homogeneity, \eqref{E:homogeneity.of.hi}, of $gh_{i}$ follows from that 
of $h_{i}$ and linearity of $g_{\ast}$.

$y$ must belong to some $\mcl{V}'$. So $C[y] \subset C[\mcl{V}']$. Therefore, by \eqref{E:pi.-1.V.=.union.of.As}, there is some $\mcl{A}_{j} \subset \mcl{V}'$ s.t.\ 
$gh_{i} \Bigl( y, \bigl[ (t, z) \bigr] \Bigr) \in C[y] \subset C[\mcl{A}_{j}]$. Therefore, 
$h_{j}^{-1} \left( gh_{i} \Bigl( y, \bigl[ (t, z) \bigr] \Bigr) \right) \in \mcl{A}_{j} \times \msf{L}_{j}$. 
($\mcl{A}_{j}$, $h_{j}$, and $\msf{L}_{j}$ ``belong'' to $\mcl{V}'$. So they depend not just on $j$ but also on $\mcl{V}'$.)

We prove that the $gh_{i}$'s satisfy property \eqref{E:pi.-1.V.=.union.of.As}. By \eqref{E:g.commutes.w/.set.ops} and \eqref{E:g.star.C[U]=C[g(U)]},
    \begin{align*}
        \bigcup_{j=1}^{n} gh_{j} \bigl[ g(\mcl{A}_{j}) \times \msf{CL}_{j} \bigr] 
           & = \bigcup_{j=1}^{n} g_{\ast} \bigl[ h_{j} ( \mcl{A}_{j} \times \msf{CL}_{j} ) \bigr] \\ 
           & = g_{\ast} \left[ \bigcup_{j=1}^{n} h_{j} ( \mcl{A}_{j} \times \msf{CL}_{j} ) \right] \\
           & = g_{\ast} \bigl( C[ \mcl{V} ] \bigr) = C \bigl[ g(\mcl{V}) \bigr].
    \end{align*}

Next, we show that $g h_{i}$ and $(g h_{i})^{-1}$ are Lipschitz w.r.t.\ $\xi \times \lambda_{i}$ 
and $\xi_{+}$. (See \eqref{E:metric.on.Pf.x.CL}, \eqref{E:lambda.metric.defn}. and \eqref{E:xi+.from.2.metrics}.)  By part \ref{I:h.hi.invrs.Lip} of definition \ref{D:fibering.by.cones}, $h_{i}$ has a Lipschitz constant, $K_{1} < \infty$. By \eqref{E:hi.is.bounded}, in conjunction with \eqref{E:homogeneity.of.hi}, and \eqref{E:g*.Lip.const.Kb.sqrd.xi+}, we have that $g_{\ast}$ is Lipschitz 
on $h_{i} ( \mcl{A}_{i} \times \msf{CL}_{i} )$ w.r.t.\ $\xi_{+}$. Let $K_{2} < \infty$ be a corresponding Lipschitz constant. Thus, by \eqref{E:metric.on.Pf.x.CL} and 
\eqref{E:xi.is.G-invar}, if $x, y \in g(\mcl{A}_{i})$; $s,t \in [0,1)$, and $w, z \in \msf{L}_{i}$,
	\begin{align*}
		  \xi_{+} \bigl[ g h_{i} (x, s, sw), g h_{i}(y,t,tz) \bigr] 
		    &=  \xi_{+} \Bigl( g_{\ast} h_{i} \bigl( g^{-1}(x), s, sw \bigr), 
		      g_{\ast} h_{i} \bigl( g^{-1}(y), t, tz \bigr) \Bigr) \\
		    &\leq K_{2} \xi_{+} \Bigl( h_{i} \bigl( g^{-1}(x), s, sw \bigr), 
		      h_{i} \bigl( g^{-1}(y), t, tz \bigr) \Bigr) \\    
		    &\leq K_{1} K_{2} (\xi \times \lambda_{i}) \Bigl[ \bigl( g^{-1}(x), s, sw \bigr), 
		      \bigl( g^{-1}(y), t, tz \bigr) \Bigr] \\
		    &= K_{1} K_{2} \Bigl[ \xi \bigl( g^{-1}(x), g^{-1}(y) \bigl) 
		      + \bigl| (s, sw) - (t, tz) \bigr| \Bigr] \\
		    &= K_{1} K_{2} \Bigl[ \xi (x, y) + \bigl| (s, sw) - (t, tz) \bigr| \Bigr] \\
		    &= K_{1} K_{2} (\xi \times \lambda_{i}) \bigl[ (x, s, sw), (y, t, tz) \bigr].
	  \end{align*}
This concludes the proof that $g h_{i}$ is Lipschitz w.r.t.\ $\xi_{+}$ and 
$\xi \times \lambda_{i}$.

Similarly, let $K < \infty$ be a Lipschitz constant for $h_{i}^{-1}$. Then, by \eqref{E:metric.on.Pf.x.CL}, \eqref{E:lambda.metric.defn}, and \eqref{E:xi.is.G-invar}, \eqref{E:hi.is.bounded}, and \eqref{E:g*.Lip.const.Kb.sqrd.xi+}, for $x$, $y$, etc. as before, 
	\begin{align*}
		(\xi \times \lambda_{i}) \bigl[ (x, s, sw), (y, t, tw) \bigr] 
		 &= \xi(x, y) + \bigl| (s,sw) - (t,tw) \bigr| \\
		 &= \xi \bigl[ g^{-1}(x), g^{-1}(y) \bigr] + \bigl| (s,sw) - (t,tw) \bigr|  \\
		 &= (\xi \times \lambda_{i}) \Bigl[ \bigl( g^{-1}(x), s, sw \bigr), 
		   \bigl( g^{-1}(y), t, tz) \bigr) \Bigr] \\
		 &\leq K \xi_{+} \Bigl[ h_{i} \bigl( g^{-1}(x), s, sw \bigr), 
		   h_{i} \bigl( g^{-1}(y), t, tz) \bigr) \Bigr] \\
		 &\leq K_{2} K \xi_{+} 
		     \Bigl[ g_{\ast} h_{i} \bigl( g^{-1}(x), s, sw \bigr), 
		       g_{\ast} h_{i} \bigl( g^{-1}(y), t, tz) \bigr) \Bigr] \\
	         &= K_{2} K \xi_{+} \bigl[ g h_{i}(x, s, sw), g h_{i}(y, t, tw) \bigr].
	\end{align*}
 
 Thus, we get a new local trivialization of $C[\Pf]$ of the form 
 $\Bigl\{ \bigl( g(\mcl{V}), g(\mcl{A}_{i}), gh_{i}, \msf{L}_{i} \bigr) \Bigr\}$. 
 \end{remark}
 
 \begin{remark}[Extending $h_{i}$] \label{R:extending.hi} 
Let $\mcl{V}$, $n$, $\mcl{A}_{i}$, $\msf{L}_{i}$, and $h_{i}$ 
($i = 1, \ldots,n$) be as in part \ref{I:local.triv} of definition \ref{D:fibering.by.cones}. Thus, $h_{i} : \mcl{A}_{i} \times \msf{CL}_{i} \to C[\Pf]$. By \eqref{E:CL.in.Euc.space}, 
$\msf{CL}_{i}$ is identified with 
$[0,1) \cdot \bigl( \{1\} \times \msf{L}_{i} \bigr) := \bigl\{ (s, sz) \in \RR^{J+1} : 0 \leq s < 1 \text{ and } z \in \msf{L}_{i} \bigr\}$. We can extend $h_{i}$ to 
$\mcl{A}_{i} \times \Bigl[ \RR \cdot \bigl( \{1\} \times \msf{L}_{i} \bigr) \Bigr]$ as follows. 
Let $y \in \mcl{A}_{i}$, $r \geq 0$, and $z \in \msf{L}_{i}$. Let $s \in (0,1)$. 
Define 
    \begin{equation}  \label{E:hji(y,r,rz).defn}
      h_{i}(y, r, rz) := (r/s) h_{i}(y, s, sz) . 
    \end{equation}
(Does this usefully make sense if $r < 0$?) $h_{i}(y, r, rz)$ is well-defined. For let $s' \in (0, 1)$. WLOG $s' \leq s$. Then, by \eqref{E:homogeneity.of.hi},  
    \begin{equation*}
      \frac{r}{s'} h_{i}(y, s', s'z) = \frac{r}{s'} h_{i} \left( y, \frac{s'}{s}s, 
        \frac{s'}{s}sz \right) = \frac{r}{s'} \frac{s'}{s} h_{i}(y, s, sz).
    \end{equation*} 
Moreover, this definition of $h_{i}$ is compatible with the original: If $r \in [0,1)$, 
take $s \in (r,1)$. 
The extended $h_{i}$ obviously has property \ref{I:pi.h.(y,w)=y} of definition \ref{D:fibering.by.cones}, but there is no guarantee that $h_{i}(y, r, rz) \in C[y]$. 

We show that \eqref{E:compatibility.across.As} remains valid in this extended setup. Suppose that for some $\mcl{V}$ and $x \in \mcl{A}_{ \ell} \cap \mcl{A}_{ m}$; $s, s' \in [0,\infty)$; 
$z \in  \msf{L}_{\ell}, \; z' \in  \msf{L}_{m} $; $\ell, m = 1, \ldots, n_{\mcl{V}}$ belonging 
to $\mcl{V}$ we have $h_{ \ell} \bigl( x, s \, (1,z) \bigr) = h_{ m} \bigl( x, s' \, (1,z') \bigr)$. Then for some $t \in (0,1)$, 
    \begin{equation*}
      (s/t) h_{\ell}(y, t, tz) = (s'/t) h_{m}(y, t, tz') .
    \end{equation*}
WLOG $s \leq s'$. So
    \begin{equation*}
      h_{\ell} \bigl( y, (s/s')t, (s/s')tz \bigr) = (s/s') h_{\ell}(y, t, tz) = h_{m}(y, t, tz') .
    \end{equation*}
Thus, by \eqref{E:compatibility.across.As}, $(s/s')t = t$. Since $t > 0$, $s = s'$ as desired. Therefore, \eqref{E:compatibility.across.As} becomes 
   \begin{multline}  \label{E:extended.compatibility.across.As}
      \text{If } h_{ \ell} \bigl( x, s \, (1,z) \bigr) = h_{ m} \bigl( x, s' \, (1,z') \bigr)
        \text{ then } s = s', \\
          x \in \mcl{A}_{ \ell} \cap \mcl{A}_{ m}; \; s, s' \in [0,\infty) ; 
            \; z \in  \msf{L}_{\ell}, \; z' \in  \msf{L}_{m} ; \; 
              \ell, m = 1, \ldots, n .
    \end{multline}

Recall \eqref{E:pi.is.proj}. It is immediate from \eqref{E:hji(y,r,rz).defn} and \eqref{E:vector.ops.on.TD} that an extended version of part \ref{I:pi.h.(y,w)=y} of definition \ref{D:fibering.by.cones} holds. Namely, 
    \begin{equation}  \label{E:pi.circ.hi.extended}
      \pi \bigl[ h_{i}(y, r, rz) \bigr] = y ,
        \quad (y \in \mcl{A}_{i}, \, r \geq 0, \text{ and } z \in \msf{L}_{i}) . 
    \end{equation}

The extended $h_{i}$ is obviously continuous. It is also homogeneous (see \eqref{E:homogeneity.of.hi.plus}) in an extended sense:
	   \begin{multline} \label{E:extended.homogeneity.of.hi}
		h_{i} \bigl[ y, (st, st z) \bigr]
		  = \frac{st}{r} h_{i} \bigl[ y, (r, r z) \bigr] 
		    = s \frac{t}{r} h_{i} \bigl[ y, (r, r z) \bigr]
		      = s h_{i} \bigl[ y, (t, t z) \bigr] , \\
		      y \in \Pf, s, t \in [0,\infty), r \in (0,1), z \in \msf{L}_{i}.
	   \end{multline}

\emph{Claim:} $h_{i}$ thus extended is one-to-one on
    \begin{equation*}
        \mcl{A}_{i} \times  \Bigl[ [0,\infty) \bigl( \{1\} \times \msf{L}_{i} \bigr) \Bigr] 
          := \bigl\{ (y, r, rz) \in \mcl{A}_{i} \times \RR^{J+1} : r \geq 0 
            \text{ and } z \in \msf{L}_{i} \bigr\}. 
    \end{equation*}
By part \ref{I:L.d-p-1} of definition \ref{D:fibering.by.cones}, to see this it suffices to show the following. Let $y, y'; \in \mcl{A}_{i}$; $r,r' \geq 0$; $z,z' \in \msf{L}_{i}$. 
We need to show that 
    \begin{equation} \label{E:hi's.equal.means.primes.equal}
      \text{\emph{if} } h_{i}(y, r, rz) = h_{i}(y', r', r'z') 
        \text{ \emph{then} } y' = y, \, r' = r, \text{ and } z' = z .
    \end{equation}
Suppose $h_{i}(y, r, rz) = h_{i}(y', r', r'z')$. If $r, r' \in [0,1)$ then \eqref{E:hi's.equal.means.primes.equal} is just part \ref{I:L.d-p-1} of definition \ref{D:fibering.by.cones}. So assume at least one of $r, r' \geq 1$, WLOG  
$r \geq 1$. Let $s \in (0,1)$. Since $s \in (0,1)$, by definition \ref{D:fibering.by.cones} part \ref{I:L.d-p-1} and \eqref{E:CL.in.Euc.space}, we may write $(y,v) := h_{i}(y, s, sz) \in C[\mcl{A}_{i}]$, $(y',v') := h_{i}(y', s, sz') \in C[\mcl{A}_{i}]$. Then, by \eqref{E:hji(y,r,rz).defn}, $h_{i}(y, r, rz) = (r/s)(y,v) = \bigl( y, (r/s) v \bigr)$ and 
$h_{i}(y', r', r'z') = (r'/s)(y',v') = \bigl( y', (r'/s) v' \bigr)$. 
Since $h_{i}(y, r, rz) = h_{i}(y', r', r'z')$ by assumption, it follows that 
    \begin{equation*}
      \bigl( y, (r/s) v \bigr) = \bigl( y', (r'/s) v' \bigr) . 
        \text{ So } y = y' \text{ and }(r/s) v = (r'/s) v' .
    \end{equation*}
Suppose $r' = 0$. Then $(r'/s) v' = 0$, which means $(r/s) v = 0$. But $r \geq 1 > 0$ 
and $s > 0$. Therefore, by \eqref{E:|h|=0.condition}, $(r/s) v \neq 0$. Contradiction. Therefore, $r' > 0$. Hence, we may assume $s \in \bigl( 0, \min \{ r, r', 1 \} \bigr)$.
Now, $s/r'$ and $s/r$ are positive and less than 1. We already know that 
$y' = y$ so we are assuming $h_{i}(y, r, rz) = h_{i}(y, r', r'z')$. Hence, by \eqref{E:hji(y,r,rz).defn}, 
    \begin{multline*}
        h_{i} \bigl (y, s^{2}/r', (s^{2}/r')z \bigr) = \frac{s^{2}}{rr'} \frac{r}{s} h_{i}(y, s, sz) \\
          = \frac{s^{2}}{rr'} h_{i}(y, r, rz) 
            = \frac{s^{2}}{rr'} h_{i}(y, r', r'z') \\
            =  \frac{s^{2}}{rr'} \frac{r'}{s} h_{i}(y, s, sz') 
              = h_{i} \bigl( y, s^{2}/r, (s^{2}/r)z \bigr) . 
    \end{multline*} 
But $s^{2}/r', s^{2}/r \in [0,1)$. Thus, $h_{i} \bigl( y, s^{2}/r', (s^{2}/r') z \bigr), h_{i} \bigl( y, s^{2}/r, (s^{2}/r) z' \bigr) \in C[y]$. Hence, by part \ref{I:L.d-p-1} of the definition, we have $\bigl( s^{2}/r', (s^{2}/r') z \bigr) = \bigl( s^{2}/r, (s^{2}/r) z' \bigr)$. I.e., $r' = r$ and $z' = z$ as desired. This proves the claim that $h_{i}$ is one-to-one 
on $\mcl{A}_{i} \times  \Bigl[ [0,\infty) \bigl( \{1\} \times \msf{L}_{i} \bigr) \Bigr]$. 
So $(y, r, rz) \neq (y', r', r'z')$ implies $h_{i}(y, r, rz) \neq h_{i}(y', r', r'z')$, but if $r \geq 1$ or $r' \geq 1$ it still might be the case that  $Exp \circ h_{i}(y, r, rz) = Exp \circ h_{i}(y', r', r'z')$. See part \ref{I:Exp.alpha.homeom} of definition \ref{D:fibering.by.cones}.

Next, we show that 
    \begin{equation} \label{E:h(y,(1,z)).notin.C[P]} 
      \text{If } y \in \mcl{A}_{i}, \, r \geq 1, \text{ and } z \in \msf{L}_{i} 
        \text{ then } h_{i} \bigl( y, r(1,z) \bigr) \notin C[\Pf]. 
    \end{equation}
For suppose $h_{i} \bigl( y, (r,z) \bigr) \in C[\Pf]$ with $r \geq 1$. Then, by \eqref{E:pi.-1.V.=.union.of.As}, and \eqref{E:CL.in.Euc.space}, there exists $j$, $s \in [0,1)$, 
and $z' \in \msf{CL}_{j}$ s.t.\ $h_{i} \bigl( y, r (1,z) \bigr) = h_{j} \bigl( y, s(1,z') \bigr)$. By \eqref{E:extended.compatibility.across.As}, this means $1 > s = r \geq 1$, contradiction. 
This proves \eqref{E:h(y,(1,z)).notin.C[P]}. 

By \eqref{E:hi's.equal.means.primes.equal}, the inverse of the extended $h_{i}$ exists on the image of $h_{i}$. Following the pattern of \eqref{E:vector.ops.on.TD}, for $y \in \mcl{A}_{i}$, $z \in \msf{L}_{i}$, and $r, r' \geq 0$ define 
$r \bigl( y, r'(1, z) \bigr) := \bigl( y, rr'(1, z) \bigr)$. 
Write $(y, v) := h_{i}\bigl( y, r'(1, z) \bigr)$. 
Then, by \eqref{E:extended.homogeneity.of.hi}, 
	\begin{multline}  \label{E:homogen.of.extndd.h.invrs}
		h_{i}^{-1} \bigl[ r(y, v) \bigr] = h_{i}^{-1} \Bigl( r h_{i}\bigl( y, r'(1, z) \bigr) \Bigr) 
		  = h_{i}^{-1} \Bigl( h_{i} \bigl( y, r r'(1, z) \bigr) \Bigr) \\
		    = \bigl( y, r r'(1, z) \bigr) = r \bigl( y, r'(1, z) \bigr)
		      = r h_{i}^{-1} (y, v) , \quad r, r' \geq 0 .
	\end{multline}
 So the extended $h_{i}^{-1}$ is also homogeneous.
 
Let $t \in [1, \infty)$ be fixed. We show that $h_{i}$, as extended, is Lipschitz on 
 $\mcl{A}_{i} \times \Bigl[ [0,t] \cdot \bigl( \{1\} \times \msf{L}_{i} \bigr) \Bigr]$. 
$t \geq 1$ and, by defintion \ref{D:fibering.by.cones} part \ref{I:L.d-p-1}, $\msf{L}_{i}$ is compact. Hence, by example \ref{Ex:ratnl.fns.loc.Lip} the map 
$s (1,z) \mapsto \tfrac{s}{2t} (1,z)$ 
($s \in [0,t]$, $z \in \msf{L}_{i}$) is Lipschitz. 
By definition \ref{D:fibering.by.cones} part \ref{I:h.hi.invrs.Lip}, $h_{i}$ (unextended) is Lipschitz. Hence, by \eqref{E:comp.of.Lips.is.Lip}, the map 
 $\bigl( y, s(1, z) \bigr) \mapsto 2 t \, h_{i} \left(y, \tfrac{s}{2t} (1,z) \right) 
 = h_{i} \bigl(y, s(1,z) \bigr)$ 
 ($y \in \mcl{A}_{i}$, $s \in [0,t]$, $z \in \msf{L}_{i}$) is Lipschitz. 
 
 By part \ref{I:L.d-p-1} of definition \ref{D:fibering.by.cones}, $\msf{L}_{i}$ is compact. Hence, there exists $L < \infty$ s.t.\ $|z| \leq L$ for every $z \in \msf{L}_{i}$. Let $K_{t} < \infty$ be a Lipschitz constant for $h_{i}\bigl( y, s(1, z) \bigr)$ with $y \in \mcl{A}_{i}$, $z \in \msf{L}_{i}$, and $s \in [0,t]$. Recall \eqref{E:vector.ops.on.TD}, \eqref{E:xi+.from.2.metrics},   \eqref{E:metric.on.Pf.x.CL}, and \eqref{E:lambda.metric.defn}. Then
    \begin{multline*}
      \Bigl| h_{i} \bigl( y, s(1,z) \bigr) \Bigr| 
        = \xi_{+} \bigl[ h_{i}(y, s(1,z)) , h_{i}(y, 0 \cdot (1,z)) \bigr] 
          \leq K_{t} (\xi \times \lambda) \Big( \bigl( y, s(1,z) \bigr), 
            \bigl( y, 0 \cdot (1,z) \bigr) \Bigr) \\ 
              = K_{t} \lambda \bigl( s(1,z) \bigr), 0 \cdot (1,z) \bigr) 
                = K_{t} s \sqrt{1 + |z|^{2}} \leq t K_{t} \sqrt{1 + L^{2}} < \infty . 
    \end{multline*}
Therefore,
    \begin{equation}  \label{E:extended.hi.is.bounded}
      \text{ For every $t \in [0, \infty)$ we have } \Bigl| h_{i} \bigl( y, s(1,z) \bigr) \Bigr| 
        \text{ is bounded in } y \in \mcl{A}_{i} , \; s \in [0,t] , \; z \in \msf{L}_{i} . 
    \end{equation}
This improves upon \eqref{E:hi.is.bounded}. 

Let $b := b_{i} > 0$ be as in \eqref{E:unif.lwr.bound.on(y,v)|>0.for.y.in.Ai}. 
Let $c \in [b, \infty)$ be fixed. We show that $h_{i}^{-1}$ is Lipschitz 
on $\mbf{F}_{[0, c]}[\mcl{A}_{i}]$. (See \eqref{E:boldF.[E].defns}.) 
Let $(y,v) \in \mbf{F}_{[0, c]}[\mcl{A}_{i}]$. Then there exists 
$w \in \RR^{k} \setminus \{0\}$ (see \eqref{E:D.imbedded.in.Rk}) and $s \in [0,c]$ 
s.t.\ $(y,w)  \in C[y]$ and $v = s |w|^{-1} w$. 
There exists $z \in \msf{L}_{i}$ and 
$t > 0$ s.t.\ $(y,w) = t h_{i} \bigl( (y, 1/2, (1/2) z \bigr) $. $|w|^{-1} w$ is unchanged if we take $t = 1$, so assume $(y,w) = h_{i} \bigl( y, 1/2, (1/2) z \bigr) $. 
Recall \eqref{E:vector.ops.on.TD}. Then, 
by \eqref{E:unif.lwr.bound.on(y,v)|>0.for.y.in.Ai}, $|w| > b := b_{i}$ so $b^{-1} > |w|^{-1}$.  Thus, 
$\tfrac{b}{c} s |w|^{-1} \leq \tfrac{b}{c} \tfrac{s}{b} = \tfrac{s}{c}  \leq 1$.
But $\tfrac{b}{c} v = \tfrac{b}{c} s |w|^{-1} w$. 
Therefore, by \eqref{E:tw.in.cone} and the fact that $(y,w)  \in C[y]$, we have 
$\bigl(  y, \tfrac{b}{c} v \bigr)  = \bigl(  y, \tfrac{b}{c} s |w|^{-1} w \bigr)  \in C[y]$.  
The map $(y, v) \mapsto \bigl(  y, \tfrac{b}{c} v \bigr) $ is Lipschitz. Now, by \eqref{E:homogen.of.extndd.h.invrs}, 
$h_{i}^{-1}(y,v) = \tfrac{c}{b} h_{i}^{-1} \bigl(  y, \tfrac{b}{c} v \bigr) $. 
By part \ref{I:h.hi.invrs.Lip} of definition \ref{D:fibering.by.cones}, this is Lipschitz in $v$. This completes the proof that $h_{i}^{-1}$ is Lipschitz on $\mbf{F}_{[0, c]}[\mcl{A}_{i}]$.
 \end{remark}

Note that in definition \ref{D:fibering.by.cones}\eqref{I:Exp.alpha.homeom}, 
we have $\alpha = \bigl( Exp \restriction_{C[\Pf]} \bigr)^{-1}$. Therefore, by \eqref{E:gExp=Expg.star},
	\begin{equation} \label{E:alpha.is.G-equivar} 
		   \alpha \circ g = g_{\ast} \circ \alpha, \quad g \in G.
	\end{equation}
(Note that $g_{\ast}(T_{y} \D) = T_{g(y)} \D$.) 

  \begin{remark} \label{R:pi.is.cont}
As mentioned earlier, $\pi_{C} : C[\Pf] \to \Pf$ is just the restriction of the projection 
$T \D \to \D$. Therefore, by \eqref{E:xi+.from.2.metrics}, it is Lipschitz w.r.t.\ the metrics 
$\xi_{+}$ and $\xi$. 
  \end{remark}  

  \begin{example}[When $\Pf$ is a manifold]  \label{Ex:TubeNbhdSpecialCaseOfCones}
Suppose $\Pf$ is an imbedded submanifold of $\D$ without boundary, i.e., a stratified space with just one stratum, $\Rcl := \Pf$. We show that $\Pf$ has a neighborhood in the restriction $T \D \restriction_{\Pf}$ that is fibered by cones as in definition \ref{D:fibering.by.cones}. 

Let $\epsilon_{\Pf}$ be as in the Tubular Neighborhood Theorem, proposition \ref{P:tubular.nbhd.thm}. Let 
    \begin{equation*}
      \epsilon : \Pf \to (0, \infty) 
        \text{ be a positive smooth $G$-invariant function as in \eqref{E:smooth.<.min} }
    \end{equation*} 
and let $\hat{N}^{\epsilon}$ be the open neighborhood of $\Pf$ in the normal bundle of $\Pf$ as in proposition \ref{P:tubular.nbhd.thm}. Then the Tubular Neighborhood Theorem continues to hold 
with $\epsilon_{\Pf}$ replaced by $\epsilon$. We show that $\hat{N}^{\epsilon}$ is fibered with cones as fibers in the sense of definition \ref{D:fibering.by.cones}.

Let $y \in \Pf$. Since $\Pf$ is an imbedded submanifold of $\D$, by Boothby \cite[Theorem (5.5), p.\ 78]{wmB75}, $y$ has a preferred coordinate neighborhood $(\clU, \varphi)$ in $\D$ so that if $W := \varphi(\clU)$ and $\mcl{V} := \clU \cap \Pf$, 
then $\varphi(y) = 0$, $W \subset \RR^{d}$ is a bounded cube, the restriction 
$\varphi \restriction_{\mcl{V}}$ maps $\mcl{V}$ onto $W \cap \RR^{p}$ (the ``waist'' of $W$, another bounded cubed), and $(\mcl{V}, \varphi \restriction_{\mcl{V}})$ is a coordinate neighborhood of $y$ in $\Pf$. Any compact subset of $\Pf$ is covered by finitely many such coordinate neighborhoods. 
Let $\psi := \varphi^{-1} : W \to \clU$ parametrize $\clU$. 
We may assume that $\overline{\clU}$, the closure of $\clU$ in $M$ is compact and $\varphi$ is the restriction of a coordinate map about $y$ on an open neighborhood of $\overline{\clU}$. \emph{A fortiori,} $\overline{\mcl{V}}$ is compact. Hence, by corollary \ref{C:cont.diff.=.loc.Lip} and \eqref{E:local.Lip.is.Lip.on.compacts}, 
$\varphi$ and $\varphi \restriction_{\mcl{V}}$ are bi-Lipschitz. For the same reason, 
$\epsilon(\cdot)$ is bounded below on $\mcl{V}$.

Let $w_{1}, \ldots, w_{d}$ be the coordinates on $\RR^{d}$.  Recall \eqref{E:superscript.perp.notation}. Pick $a_{ij}$ 
($i = 1, \ldots, d-p; \; j = 1, \ldots, d$) 
s.t.\ $X_{y,i} := \sum_{j = 1}^{d} a_{ij} \psi_{\ast} (\partial/\partial w_{j} 
\restriction_{w = \varphi(y)} ) \in T_{y} \D$ 
($i = 1, \ldots, d-p$) is an orthonormal basis of $(T_{y} \Pf)^{\perp}$. Here, $y \in \mcl{V} \subset \Pf$ is our given point.
Now let $x$ vary in $\mcl{V}$ but, using the same constant $a_{ij}$'s, define 
$X_{x,i} := \sum_{j = 1}^{d} a_{ij} \psi_{\ast} (\partial/\partial w_{j} \restriction_{w = \varphi(x)} ) \in T_{x} \D$ 
($i = 1, \ldots, d-p$). So $X_{1}, \ldots, X_{d-p}$ are smooth vector fields on $\mcl{V}$. Now, if $x \in \mcl{V} \setminus \{y\}$ it might not be the case that $X_{x,1}, \ldots, X_{x,d-p}$ is an orthonormal basis of $(T_{x} \Pf)^{\perp}$. Let $Y_{x,i}$ be the orthogonal projection of $X_{x,i}$ onto $(T_{x} \Pf)^{\perp}$ so $Y_{y,i} = X_{y,i}$ 
($i = 1, \ldots, d-p$) 
and $X_{1}, \ldots, X_{d-p}$ are smooth vector fields on $\mcl{V}$. 
Since $Y_{y,i} = X_{y,i}$ and $X_{y,i}$ ($i = 1, \ldots, d-p$) is an orthonormal basis 
of $(T_{y} \Pf)^{\perp}$, by making 
$\mcl{V}$ smaller if necessary, we may assume that for every $x \in \mcl{V}$ the tangent vectors $Y_{x,1}, \ldots, Y_{x,d-p}$ are linearly independent vectors 
in $(T_{x} \Pf)^{\perp}$. Finally, apply Gram-Schmidt (Stoll and Wong 
\cite[Theorem 2.2, p.\ 73]{rrSetW68.LinearAlgebra}) to $Y_{x,1}, \ldots, Y_{x,d-p}$ to get smooth orthonormal vector fields 
$V_{x,1}, \ldots, V_{x,p-d} \in (T_{x} \Pf)^{\perp}$ ($x \in \mcl{V}$). Thus, the $V_{x,i}$'s and $X_{x,i}$'s agree at $x = y$. 
Write $V_{x,i} = (x, v_{x,i})$, where $v_{x,i} \in \RR^{k}$. (See \eqref{E:D.imbedded.in.Rk}.)
Then the $v_{x,i}$'s are orthonormal. Moreover, by corollary \ref{C:cont.diff.=.loc.Lip} again, relative compactness of $\mcl{V}$, and \eqref{E:local.Lip.is.Lip.on.compacts}, 
$v_{x,1}, \ldots, v_{x,p-d}$ are Lipschitz in $x \in \mcl{V}$. Similarly, by \eqref{E:smooth.<.min}, 
    \begin{equation}  \label{E:eps.on.manif.is.Lip.bnd.below}
      \epsilon \text{ is Lipschitz and bounded below on } \mcl{V} .
    \end{equation}

Define: 
    \begin{equation}  \label{E:C[x].for.manif}
      C[x] \text{ is the collection of all } X \in T_{x} \D (x \in \mcl{V}) \text{ s.t.\ }
        X \perp T_{x} \Pf, \text{ and } |X| < \epsilon(x) .
    \end{equation}

Let $\msf{L}_{1} := S^{d-p-1}$, the unit $(d-p-1)$-sphere in $\RR^{d-p}$ centered at 0. 
$\msf{L}_{1}$ is a manifold, i.e., a stratified space with only one stratum. By part \ref{I:cmpct.manif.finitely.tractble} of lemma \ref{L:strat.properties}, $\msf{L}_{1}$ satisfies \eqref{E:tameness.of.stratified.space}.
For $z = (z_{1}, \ldots, z_{d-p}) \in \msf{L}_{1}$ and $x \in \mcl{V}$, 
let $z \cdot v_{x} := \sum_{j=1}^{d-p} z_{j} v_{x,j}  \in \RR^{k}$. 
Thus, $(x, z \cdot v_{x}) \perp T_{x} \Pf$. Recall \eqref{E:vector.ops.on.TD}. 
If $x \in \mcl{V}$, $z \in \msf{L}_{1}$, and $s \in \bigl[ 0, \epsilon(x) \bigr)$ then 
    \begin{equation*}
      (x, s \, z \cdot v_{x}) = \sum_{j=1}^{d-p} z_{j} V_{x,j} \in T_{x} \Pf . 
    \end{equation*}
Clearly, 
    \begin{equation}  \label{E:C[x].on.manif}
      C[x] = \bigl\{ (x, s \epsilon(x) \, z \cdot v_{x}) : 0 \leq s < 1 
        \text{ and } z \in \msf{L}_{1} \bigr\} .
    \end{equation}
 
Let $\mcl{A}_{1} := \mcl{V}$, so $\mcl{A}_{1}$ is open in $\Rcl := \Pf$ and closed 
in $\mcl{V}$. Since $\Pf$ has only one stratum, itself, 
so does $\mcl{A}_{1}$, viz. $\mcl{A}_{1} \cap \Pf = \mcl{A}_{1}$ as required by part \ref{I:local.triv} of definition \ref{D:fibering.by.cones}. 
Thus, by part \ref{I:single.stratum.space.is.tame} of lemma \ref{L:strat.properties}, 
$\mcl{A}_{1}$ satisfies \eqref{E:tameness.of.stratified.space}. Hence, part \eqref{I:L.A.tameness} of definition \ref{D:fibering.by.cones} holds for $\mcl{A}_{1}$ and 
$\msf{L}_{1}$.     

Let $w = s(1,z) = \bigl[ (s,z) \bigr] \in \msf{CL}_{1}$, where $s \in [0,1)$, 
$z \in \msf{L}_{1}$ (see \eqref{E:CL.in.Euc.space}) and let $x \in \mcl{A}_{1}$. Define
    \begin{equation*}
      h_{1}(x, w) := \bigl( x, s \epsilon(x) \, z \cdot v_{x} \bigr) \in C[\mcl{A}_{1}] 
        \subset (T_{x} \D)^{\perp} .
    \end{equation*}
Thus, $h_{1}$ is an injection, in fact by \eqref{E:C[x].on.manif}, a bijection and \eqref{E:pi.-1.V.=.union.of.As} holds with $n=1$. Moreover, by \eqref{E:eps.on.manif.is.Lip.bnd.below}, example \ref{Ex:ratnl.fns.loc.Lip}, \eqref{E:comp.of.Lips.is.Lip}, and the fact that 
$v_{x,1}, \ldots, v_{x,p-d}$ are Lipschitz in $x \in \mcl{V}$, the map 
$\bigl( x, s(1,z) \bigr) \mapsto \bigl( x, s \epsilon(x) \, z \cdot v_{x} \bigr)$
is Lipschitz. 

Projection onto the first factor of $h_{1}(x, w)$ is a Lipschitz map yielding $x$. Projection onto the second factor is a Lipschitz map yielding 
$s \epsilon(x) \, z \cdot v_{x}$. By \eqref{E:eps.on.manif.is.Lip.bnd.below}, $\epsilon$ is Lipschitz and bounded below on $\mcl{A}_{1} := \mcl{V}$. 
In addition, $|s| \leq 1$, $\epsilon$ is continuous on $\mcl{A}_{1}$ and $\mcl{A}_{1}$ is relatively compact so $\epsilon$ is bounded on $\mcl{A}_{1}$, and $|v_{x}| = 1$. 
Therefore, $s \epsilon(x) \, z \cdot v_{x}$ is bounded on $\mcl{A}_{1}$. 
Hence, by example \ref{Ex:ratnl.fns.loc.Lip} again, 
$s \epsilon(x) \, z \cdot v_{x} \mapsto \epsilon(x)^{-1} 
\bigl( s \epsilon(x) \, z \cdot v_{x} \bigr) = s \, z \cdot v_{x}$ is Lipschitz in $x \in \mcl{A}_{1}$. 
Since $v_{x,1}, \ldots, v_{x,p-d}$ are orthonormal and Lipschitz in $x \in \mcl{A}_{1}$, the map $s \, z \cdot v_{x} \mapsto sz$ is Lipschitz. Finally, $s = |sz|$ is Lipschitz in $sz$, by\eqref{E:norm.is.Lip}. Putting all this together, it follows from \eqref{E:comp.of.Lips.is.Lip} that $h_{1}^{-1} : \bigl( x, s \epsilon(x) \, z \cdot v_{x} \bigr) \mapsto \bigl(x, s(1,z) \bigr) \in \mcl{A}_{1} \times \msf{L}_{1}$ is Lipschitz. Thus,  
$h_{1}$ is bi-Lipschitz. Hence, parts \ref{I:L.d-p-1}, \ref{I:CV.from.h.Ai}, and \ref{I:h.hi.invrs.Lip} of definition \ref{D:fibering.by.cones} hold. Part \ref{I:pi.h.(y,w)=y} holds trivially.

By \eqref{E:vector.ops.on.TD}, $h_{1}$ satisfies property \eqref{I:homogeneity.of.hi} of definition \ref{D:fibering.by.cones}. Thus, $\Pf$ satisfies part \ref{I:local.triv} of definition \ref{D:fibering.by.cones}. 

We prove that property \ref{I:Exp.alpha.homeom} holds. 
That $Exp \bigl( C[\Pf] \bigr) = Exp \bigl( \hat{N}^{\epsilon} \bigr)$ is open is immediate from \eqref{E:C[x].for.manif} and \eqref{E:tubular.N.eps.defn}.
By \ref{P:tubular.nbhd.thm}, $Exp^{-1}$ exists and both $Exp$ and $Exp^{-1}$ are diffeomorphisms on $C[\Pf]  = \hat{N}^{\epsilon}$. Let $r \in (0,1)$. Replace $\epsilon$ by $r \epsilon$ in the definition of $C[\Pf]$. Let $\mcl{K} \subset \Pf$ be compact. Then $C[\mcl{K}]$ is relatively compact in $T \D$. Then the Lipschitz property follows from from proposition \ref{P:tubular.nbhd.thm}, corollary \ref{C:cont.diff.=.loc.Lip}, and \eqref{E:local.Lip.is.Lip.on.compacts}. 

Finally, we prove \eqref{E:C.C[P].G.invar}. 
Let $(y, v) \in C[y]$ and let $g \in G$. By assumption in definition \ref{D:fibering.by.cones}, $g(y) \in \Pf$. By \eqref{E:tube.nbhd.G-invar}, 
$g_{\ast}(y,v) = \bigl( g(y), w \bigr) \in (T_{g(y)} \Pf)^{\perp}$ for some $w \in \RR^{k}$. By \eqref{E:g.preserves.Riem.met} and \eqref{E:Riem.metrics.on.D.on.Rk.same}, we have $|v| = \| (y,v) \|_{y} = \| g_{\ast}(y,v) \|_{g(y)} = |w|$. Therefore, since $\epsilon$ is $G$-invariant, $\epsilon \bigl[ g(y) \bigr] = \epsilon(y) > |w|$, 
so $g_{\ast}(y,v) = \bigl( g(y), w \bigr) \in C \bigl[ g(y) \bigr]$. 
I.e., $g_{\ast} \bigl( C [y] \bigr) \subset C[y]$. Since $g \in G$ is arbitrary, we must also have 
$g^{-1}_{\ast} \Bigl( C \bigl[ g(y) \bigr] \Bigr) \subset C [y]$. 
I.e., $g_{\ast} \bigl( C [y] \bigr) = C[y]$. 

Now suppose $x \in \mcl{C}$. Then there exists $(y,v) \in C[\Pf]$ s.t.\ $x = Exp(y,v)$. By \eqref{E:gExp=Expg.star}, $g(x) = g \circ Exp(y,v) = Exp \circ g_{\ast}(y,v)$. But we have just seen that $g_{\ast}(y,v) \in C[\Pf]$. Thus, $g(x) \in Exp \bigl( C[\Pf] \bigr) = \mcl{C}$, i.e.\ $g(\mcl{C}) \subset \mcl{C}$. By the same token, $g^{-1}(x) \in \mcl{C}$, 
i.e.\ $\mcl{C} \subset g(\mcl{C})$ and \eqref{E:C.C[P].G.invar} is proved.
  \end{example}
  
Another example, used in chapter \ref{Chptr:robst.loc.on.circle}, is developed in appendix \ref{Chptr:rob.loc.circle.cones.appendix}.

See appendix \ref{Chptr:misc.proofs} for the proof of the following. You may wish to review definition \eqref{E:boldF.[E].defns}. Recall that a cover of $\Pf$ is ``locally finite'' if every point of $\Pf$ has a neighborhood that intersects only finitely many sets in the cover
     \begin{lemma} \label{L:2eps.P.exists} 
Suppose a neighborhood of $\Pf$ in $T \D \restriction_{\Pf}$ is fibered
over $\Pf$ with cone fibers $C[x]$, $x \in \Pf$. Let $\mcl{E}_{1}, \mcl{E}_{2}, \ldots$ be a locally finite open cover of $\Pf$. Let $t_{1}, t_{2}, \ldots \in (0, \infty]$. Then there exists a continuous function $\epsilon : \Pf \to (0, \infty)$ s.t.\ $\epsilon$ has a $C^{\infty}$ extension to $\mcl{C}$  
and for every $x \in \Pf$ we have 
$\mbf{F}_{\bigl[0, 2 \, \epsilon(x) \bigr]}[x] \subset C[x]$ and, hence,
$Exp \Bigl( \mbf{F}_{\bigl[0, 2 \, \epsilon(x) \bigr]}[x] \Bigr) \subset \mcl{C}$. 
Moreover, for every $i = 1, 2, \ldots$, if $x \in \mcl{E}_{i}$ 
then $2 \, \epsilon(x) < t_{i}$. We may assume $\epsilon$ is $G$-invariant: 
$\epsilon \circ g = \epsilon$ for every $g \in G$. 
    \end{lemma}
Here, $Exp \Bigl( \mbf{F}_{\bigl[0, 2 \, \epsilon(x) \bigr]}[x] \Bigr) 
  := \Bigl\{ Exp(x,v) : (x,v) \in \mbf{F}_{\bigl[0, 2 \, \epsilon(x) \bigr]}[x] \Bigr\}$.
From now on, use the symbol $\epsilon_{\Pf}$ to refer to some fixed function $\epsilon$ as in the lemma corresponding 
to the trivial cover $\{ \mcl{E}_{1} \} := \{ \Pf \}$ and $t_{1} = + \infty$. Then
	\begin{equation} \label{E:eps.P.properties}
	         \mbf{F}_{\bigl[0, 2 \, \epsilon_{\Pf}(x) \bigr]}[x] \subset C[x]
		  \text{ and } \epsilon_{\Pf} \circ g = \epsilon_{\Pf} \text{ for every } x \in \Pf 
		    \text{ and } g \in G.
	\end{equation}
Note that, since $\epsilon_{\Pf}$ has a $C^{\infty}$ extension to an open neighborhood of $\Pf$, we have that $\epsilon_{\Pf}$ is smooth on each stratum of $\Pf$.

  \begin{example}[Locally finite cover]  \label{Ex:nested.compacts}
Let $\X$ be a (nonempty) locally compact second countable Hausdorff space. There exists a sequence $\mcl{K}_{0} = \varnothing, \mcl{K}_{1}, \mcl{K}_{2}, \ldots$, possibly finite, of compact subsets of $\X$ whose union is $\X$ and satisfying 
$\mcl{K}_{i-1} \subset \mcl{K}_{i}^{\circ}$ ($i = 1, 2, \ldots$; 
Ash \cite[Theorem A5.15, p.\ 387]{rbA72}; $\mcl{K}_{i}^{\circ}$ is the interior 
of $\mcl{K}_{i}$ in $\X$). Since $\varnothing^{\circ} = \varnothing$, it is possible 
for $\mcl{K}_{1} = \varnothing$. But, by droping all empty $\mcl{K}_{i}$s with positive $i$ and relabeling we may assume $\mcl{K}_{1} \neq \varnothing$. 

Let $x \in \X$ and let $j = i > 0$ be the smallest $j$ s.t.\ $x \in \mcl{K}_{j}$. I.e., $x \in \mcl{K}_{i} \setminus \mcl{K}_{i-1}$. Notice that $i$ is uniquely determined by $x$. 
Now, $\mcl{E}_{i} := \mcl{K}_{i+1}^{\circ} \setminus \mcl{K}_{i-1}$ is a relatively compact  open set containing 
$\mcl{K}_{i} \setminus \mcl{K}_{i-1}$. Thus, the collection $\{ \mcl{E}_{i} \}$ is an open, relatively compact cover of $\X$. Clearly, it is locally finite. Some of the sets $\mcl{E}_{i}$ may be empty. 
If $\mcl{E}_{i}$ is empty then $\mcl{K}_{i+1}$ is both open and closed. 
Since $\mcl{K}_{i+1} \neq \varnothing$, then it is  
a union of connected components of $\X$, a finite union if $\X$ is locally connected.
So $\mcl{K}_{i+1} = \X$ is possible. Eliminate empty $\mcl{E}_{i}$s. 

If $\X$ is compact then eventually $\mcl{K}_{i} = \X$, in which case either the sequence 
$\mcl{K}_{0} = \varnothing, \mcl{K}_{1}, \mcl{K}_{2}, \ldots$ terminates 
or $\mcl{K}_{i} = \X$ from some point on (possible since $\X^{\circ} = \X$ relative to itself).
  \end{example}

Say that the fibering over $\Pf$ has ``relatively compact trivialization'' if all the sets $\mcl{A}_{i}$ are relatively compact in $\Pf$. (Perhaps a fibering by cones can always be made to have a relatively compact trivialization. I don't know.) 

The following is an application of lemma \ref{L:2eps.P.exists}. For proof see appendix \ref{Chptr:misc.proofs}.
  \begin{lemma}  \label{L:rescaling.C[]}
Suppose the fibering over $\Pf$ has relatively compact trivialization. 
Let $\{ \mcl{E}_{j}, j = 1, 2, \ldots \}$ be a locally finite open covering of $\Pf$. For each $j$, let $t_{j} \in (0, \infty]$. Then we may assume that 
    \begin{equation}  \label{E:v.short}
      (y, v) \in C[\mcl{E}_{j}] \text{ implies that } |v| < t_{j}. 
    \end{equation}
I.e., if $\bigl\{ C[y] \subset T_{y} \D : y \in \Pf \bigr\}$ has the properties listed in definition \ref{D:fibering.by.cones} plus relatively compact trivialization, then we can replace it by another fibering having those properties but for which in addition \eqref{E:v.short} holds.
  \end{lemma} 
In particular, suppose $\Pf$ has relatively compact trivialization. 
Then, taking $\mcl{E}_{1} := \Pf$ and $t_{1} > 0$ arbitrary, we see that 
	\begin{multline}  \label{E:|(y,v)|.bounded}
	    \text{If } \Pf \text{ has relatively compact trivialization then } \\
	    \text{we may assume } \bigl\{ |v| : (y, v) \in C[\Pf] \bigr\} \text{ is bounded.}
	\end{multline}
(See \eqref{E:C[K].rel.compact}.)

Another application of lemma \ref{L:rescaling.C[]} is the following. (See \eqref{E:xi.is.metric.on.D}.)
	\begin{multline} \label{E:|(y,v)=xi(Exp(y,v),y)}
	    \text{If } \Pf \text{ has relatively compact trivialization then } \\	
		\text{we may assume that if } (y,v) \in C[\Pf] 
		  \text{ then } |v| = \xi \bigl[ y, Exp(y,v) \bigr],
	\end{multline}
the distance from $y$ to $Exp(y,v)$. We prove this as follows. If $x \in \D$ and $r > 0$, let $\mathscr{B}_{r}(x)$ be the open ball in $\D$ about $x$ with radius $r$ 
w.r.t.\ $\xi$. Let $B_{x,r}(0)$ be the open ball about 0 in $T_{x} \D \subset \RR^{k}$ 
with radius $r$. By proposition \ref{P:geod.cnvx.nbhds.exist}, 
for every $y \in \Pf$ there exists 
$\delta_{y} \in (0,1]$ s.t.\ $\mathscr{B}_{\delta_{y}}(y)$ is geodesically convex and a normal neighborhood of each of its points. (See Boothby \cite[p.\ 335]{wmB75}.) Therefore, by 
remark \ref{R:short.geods.geod.cnvx}, 
$Exp \bigl[ B_{x,\delta_{y}}(0) \bigr] = \mathscr{B}_{r}(x)$. 
By \eqref{E:Pf.has.countable.loc.finite.refinement}, 
the cover $\bigl\{ \mathscr{B}_{\delta_{y}/2}(y), y \in \Pf \bigr\}$ has a countable locally finite refinement $\{ \mcl{E}_{i}, i = 1, 2, \ldots \}$ covering $\Pf$.

Let $i = 1, 2, \ldots$. For some $y' \in \Pf$, we have 
$\mcl{E}_{i} \subset \mathscr{B}_{\delta_{y'}/2}(y')$. Let
$s_{i} := \tfrac{1}{4} \sup \{ \delta_{y'} > 0 : y' \in \Pf \text{ and } 
\mcl{E}_{i} \subset \mathscr{B}_{\delta_{y'}/2}(y') \} \leq 1/4$.
There exists $y_{i} \in \Pf$ s.t.\ $\mcl{E}_{i} \subset \mathscr{B}_{\delta_{y_{i}}/2}(y_{i})$ and 
$\delta_{y_{i}} \geq \tfrac{2}{3} \sup \{ \delta_{y'} > 0 : y' \in \Pf \text{ and } 
\mcl{E}_{i} \subset \mathscr{B}_{\delta_{y'}/2}(y') \}$. 
Thus $\delta_{y_{i}}/2 > s_{i}$. 

Suppose $(y, v) \in C[\Pf]$ with $y \in \mcl{E}_{i} \subset \mathscr{B}_{\delta_{y_{i}}/2}(y_{i})$. Then, by lemma \ref{L:rescaling.C[]}, we may assume that $|v| < s_{i}$. 
That means the length of the shortest geodesic from $y$ to $Exp(y,v)$ is less than 
$s_{i} < \delta_{y_{i}}/2$. Therefore, $Exp(y,v) \in \mathscr{B}_{\delta_{y_{i}}/2}(y) \subset \mathscr{B}_{\delta_{y_{i}}}(y_{i})$. 
Since $\mathscr{B}_{\delta_{y_{i}}}(y_{i})$ is geodesically convex and a normal neighborhood of $y$, by remark \ref{R:short.geods.geod.cnvx}, the arc $t \mapsto Exp_{y}(t v)$ ($t \in [0,1]$) is a shortest geodesic connecting $y$ to $Exp_{y}(v)$. Hence, $\xi \bigl[ y, Exp_{y}(v) \bigr] = |v|$. I.e., \eqref{E:|(y,v)=xi(Exp(y,v),y)} holds.

Unlike in the Tubular Neighborhood Theorem case,  
\eqref{E:alpha.dist.to.P}, if $(y, v) \in C[y]$ in general we do not have 
$\text{dist} \bigl[ \alpha^{-1}(y,v), \Pf \bigr] = |v|$ if $(y, v) \in C[\Pf]$. (See figure \ref{F:ConeBundle} for an example.) However, trivially, we have
	\begin{equation}  \label{E:dist.Exp.v.to.P.bnd}
	  \text{dist} \bigl[ Exp(y,v), \Pf \bigr] \leq |v| = \xi \bigl[ Exp(y,v), y \bigr] ,
	     \quad (y, v) \in C[\Pf] .
	\end{equation}

\section{Main theorem} \label{SS:main.haus.theorem}
The following property (not proposition!) gives a precise interpretation of the ``sales pitch'', remark \ref{R:sales.pitch}. Also see remark \ref{R:extended.sales.pitch}.

   \begin{property}  \label{Pty:agree.near.T} 
Let $\T \subset \D$. The quintuplet $(\Phi, \Ss', G, \T, a)$ satisfies the following. 
  \begin{enumerate}
	\item $\Ss' \subset \D$ is closed with empty interior and 
	$\Phi : \D' := \D \setminus \Ss' \to \F$ is continuous. 
For every $x \in \D \setminus \Ss'$ and $g \in G$ we have $\Phi \circ g(x) = \Phi(x)$ (``$\Phi$ is $G$-invariant''). For every $g \in G$ we have $g(\Ss') = \Ss'$. (``$\Ss'$ is $G$-invariant").  \label{I:Phi.is.nice}
	\item Let $\Psi : \D \setminus \tilde{\Ss} \to \F$ be \emph{any} continuous $G$-invariant data map, where $\tilde{\Ss} \subset \D$ is closed with empty interior and $G$-invariant, 
	$\tilde{\Ss} \cap \T = \Ss' \cap \T$, and the restrictions 
$\Psi \restriction_{\T \setminus \Ss'}$ and $\Phi \restriction_{\T \setminus \Ss'}$ are equal. 
\emph{Then} $\Hm^{a} (\tilde{\Ss}) > 0$. In particular, with $\Psi = \Phi$, we have 
$\Hm^{a} (\Ss') > 0$ so $\dim \Ss' \geq a$. \label{E:Psi.is.nice}
  \end{enumerate}
   \end{property}

Property \ref{Pty:agree.near.T} is a precise formulation of the notion that only a data map's behavior in the immediate vicinity of (in an ``infinitesimal neighborhood of'') $\T$ matters. See remark \ref{R:sales.pitch}. 

Suppose $\Phi$ is a data map partly defined on some space $\D$
and, for some $\Ss' \subset \D$, satisfies \eqref{I:Phi.is.nice} in property \ref{Pty:agree.near.T}. Let $\T \subset \D$ and suppose 
$\Ss' \cap \T \neq \varnothing$. Then $\Phi$ has property \ref{Pty:agree.near.T} with $a = 0$. (See \eqref{E:0.dim.Haus.measure} and \eqref{E:Haus.dim.defn}.)

   \begin{remark}[Immediate consequences of property \ref{Pty:agree.near.T}]  \label{R:consequences.of.property}
Suppose $(\Phi, \Ss', G, \T, a)$ has property \ref{Pty:agree.near.T} and $(\Psi, \tilde{\Ss})$ is as in the statement of the property. 
Then $(\Psi, \tilde{\Ss}, G, \T, a)$ also has property \ref{Pty:agree.near.T}. 
If $(\Phi, \Ss', G, \T, a)$ satisfies property \ref{Pty:agree.near.T} then we may modify $\Phi$ in various ways, \emph{perhaps nonalgorithmically}, to get a data map continuous off a closed set $\tilde{\Ss}$ and still know that 
$\dim \tilde{\Ss} \geq a$, providing $\tilde{\Ss} \cap \T = \Ss' \cap \T$ and we do not modify 
$\Phi$ on $\T \setminus \Ss'$. If theorem \ref{T:lwr.bnd.on.Haus.meas} below applies to 
$\Ss'$ then it will also apply to $\tilde{\Ss}$. 

If $(\Phi, \Ss', G, \T, a)$ satisfies property \ref{Pty:agree.near.T} then $\dim \Ss' \geq a$. Hence, by appendix \ref{Chptr:Lip.Haus.meas.dim}, if we also have 
$\Hm^{a}( \Ss' ) < \infty$, then $\dim \Ss' = a$. 
   \end{remark}

  \begin{example}  \label{Ex:basic.property.case}
Suppose $\Phi : \D \partlyto \F$ is a $G$-invariant data map and the hypotheses of proposition \ref{P:sing.dim.when.H.d-r.D.=.0} hold. Then $\Phi$ satisfies property \ref{Pty:agree.near.T} with $a = d-r-1$ and any finite group $G$ s.t.\ $\Phi$ and 
$\Ss^{\msf{V}}$ are $G$-invariant. 

It follows that principal component plane-fitting has property \ref{Pty:agree.near.T}. In \cite[pp.\ 495--496]{spE95} this idea is applied to immediately show that the singular set of the plane-fitting method proposed in Friedman \cite{jHF87.ProjPursuit} has co-dimension no bigger then 2.
  \end{example}

  \begin{remark}[Stronger version of property \ref{Pty:agree.near.T}?]  \label{R:stronger.property}
In proving the main theorem of this chapter (theorem \ref{T:lwr.bnd.on.Haus.meas}) we will find ourselves focussing on integral $a$. A plausible stronger version of property \ref{Pty:agree.near.T} is obtained by replacing 
$\Hm^{a} (\tilde{\Ss}) > 0$ by $\check{H}^{a}(\tilde{\Ss}) \neq \{0\}$. (See \eqref{E:cohom.and.codim}.) In the proof of proposition \ref{P:sing.dim.when.H.d-r.D.=.0} we see that under the conditions described in remark \ref{Ex:basic.property.case}  
$(\Phi, \Ss', G, \T, a)$ has this property when $a = d-r-1$. Implications of this are discussed in remark \ref{R:topological.bounds.on.vol}.
  \end{remark}

For the rest of this chapter we assume nothing about $\Ss'$ except that it is closed. 
In chapter \ref{Chptr:severity}, we show that we may sometimes replace $\Ss'$ by the singular set, $\Ss$, of $\Phi$, which need not be closed. By \eqref{E:D.is.cmpct.Riem.manif}, we assume $\D$ is a $C^{\infty}$ manifold. (See chapter \ref{Chptr:basic.setup}.) The definition of ``triangulation'' is given in \eqref{E:triangulation.defn}. For the definition of ``$C^{\infty}$ triangulation'' see Munkres \cite[Definition 8.3, pp.\ 80--81]{jrM66}. 

  \begin{lemma}  \label{L:biLip.triangulation}
If $\D$ is a compact Riemannian $C^{\infty}$ manifold then it has a $C^{\infty}$ triangulation. Any such is ``bi-Lipschitz'' (both the triangulation 
and its inverse are Lipschitz; \eqref{E:bi-Lipschitz.defn}). The simplicial complex triangulating $\D$ is finite. In particular, if $\xi$ is the topological metric on $\D$ induced by the Riemannian metric (\eqref{E:xi.is.metric.on.D}), then there is a finite simplicial complex, $P$, with $|P| \subset \RR^{N}$ for some $N = 1, 2, \ldots$, and a homeomorphism, 
$f : |P| \to \D$ s.t.\ both $f$ and $f^{-1}$ are Lipschitz relative to $\xi$ and the metric $|P|$ inherits from $\RR^{N}$.
  \end{lemma}

See appendix \ref{Chptr:misc.proofs} for proof. The lemma does not say that the triangulation $f : |P| \to \D$ has to be $G$-invariant if $\D$ is. (By ``$G$-invariant'' we mean that $\{ f^{-1} \circ g \circ f : |P| \to |P| : g \in G \}$ is a group of simplicial homeomorphisms on $P$, definition \ref{D:simplicial.homeom}. 
Lemma \ref{L:Cart.pwrs.of.simp.cmplx.can.be.perm.invar.cmplx} provides an example of a triangulation invariant under a group action.

In this chapter we make heavy use of Hausdorff dimension and measure (appendix \ref{Chptr:Lip.Haus.meas.dim}). Use $\xi$, \eqref{E:xi.is.metric.on.D}, to compute Hausdorff measure on $\D$.

It is reasonable to suppose that in statistical data analysis $\Ss'$ will have some regularity properties, but that cannot be said of the singular sets of information processing by living organisms. Still, the generality of the following allows application to such organic data maps.. Recall the definition of a simplicial homeomorphism (definition \ref{D:simplicial.homeom}).

In conformity with chapter \ref{Chptr:basic.setup}, since $\D$ is Riemannian, we assume it is $C^{\infty}$. 
   \begin{theorem}   \label{T:lwr.bnd.on.Haus.meas}
Let $\D$ be a compact Riemannian manifold (Boothby \cite[Definition (2.6), p.\ 184]{wmB75}) 
of dimension $d$. Suppose $G$ is a finite group of diffeomorphisms on $\D$ and the Riemannian metric, $\langle \cdot, \cdot \rangle$, on $\D$ is $G$-invariant. I.e., for every $g \in G$ we have 
$g^{\ast} \bigl( \langle \cdot, \cdot \rangle \bigr) = \langle \cdot, \cdot \rangle$. Put on $\D$ the topological metric, $\xi$, corresponding to $\langle \cdot, \cdot \rangle$. (See \eqref{E:xi.is.metric.on.D}) Use $\xi$ to compute Hausdorff measure on $\D$. Suppose there is a simplicial complex, $P$ with the following property. $\D$ has a bi-Lipschitz triangulation 
$f : |P| \to \D$ s.t.\ $\{ f^{-1} \circ g \circ f : |P| \to |P| : g \in G \}$ is a group of simplicial homeomorphisms on $P$. 

Let $\T \subset \Pf \subset \D$ ($\T = \Pf$ is possible) with both $\T$ and $\Pf$ $G$-invariant. Suppose $\T \subset \D$ is a compact, smooth, imbedded submanifold 
of $\D$. Let $t \in [0, d)$ be the dimension of $\T$. Suppose $\Pf$ is a stratified space having a neighborhood in $T \D \restriction_{\Pf}$ fibered over $\Pf$ by cones as described in subsection \ref{SSS:conical.fibers}. Let $p < d := \dim \D$ be the maximum dimension of strata of $\Pf$. So $t \leq p$.  

 Let $\Phi : \D \partlyto \F$ be a data map and let $\Ss' \subset \D$ be closed 
 with $\Phi$ defined and continuous on $\D \setminus \Ss'$. Let $a \in [0, d)$ and suppose 
 $(\Phi, \Ss', G, \T, a)$ satisfies property \ref{Pty:agree.near.T}.     
Then there is a $\gamma > 0$, depending only on $\D$,  
$C[\Pf]$, $a$, and $\F$, with the following property. Suppose $R > 0$ and 
$dist^{a}(\Ss', \Pf) \geq R$, i.e.,
             \begin{equation}  \label{E:essntl.dist from.S.to.Pf.small}
                \Hm^{a} \Bigl( \bigl\{ x \in \Ss' : dist( x, \Pf ) < R \bigr\} \Bigr) = 0. 
             \end{equation} 
(See \eqref{E:essential.dist.defn}.) Then  
             \begin{equation}  \label{E:Hm.a.S.geq.R.d-p-1}
                \Hm^{a}(\Ss') \geq \gamma R^{ \min(d-p-1,a) }.
             \end{equation}  
If $\Hm^{a}(\Ss') < \infty$ then $a$ is an integer.
   \end{theorem}

In appendix \ref{Chptr:F:CircleSing.data.and.calcs}, the volumes of the singular sets of two data maps are computed and compared. However, in that calculation no use is made of \eqref{E:Hm.a.S.geq.R.d-p-1}. 

The problem of computing lower bounds for $\gamma$ in \eqref{E:Hm.a.S.geq.R.d-p-1} is discussed in remark \ref{R:recipe.for.gamma}. 

By property \ref{Pty:agree.near.T}, $\dim \Ss' \geq a$. If the inequality is strict, i.e. 
$\dim \Ss' > a$, then, by \eqref{E:Haus.dim.defn}, $\Hm^{a}(\Ss') = \infty$, rendering the inequality \eqref{E:Hm.a.S.geq.R.d-p-1} trivial. However, at least in the data maps examined in chapters \ref{Chptr:sings.in.plane.fit} through \ref{Chptr:linear.classification}, it is usually the case that $\Hm^{a}(\Ss') = a$. 

\begin{remark}[``Fit-instability'' tradeoff]  \label{R:fit.instab.tradeoff}
The inequality \eqref{E:Hm.a.S.geq.R.d-p-1} describes a fit-instability tradeoff akin to the famous a variance-bias tradeoff 
(Hastie \emph{et al} \cite[Section 2.9]{tHrTjF01.statlearn}). All else being equal, one prefers a small singular set. All else being equal, one wishes the singularities, at least the severe ones (chapter \ref{Chptr:severity}), to be far from $\Pf$. 

The first desideratum is clear, but why is the second desirable? At a data set $x \in \Pf$ there is a unique point in $\F$ to which $x$ should be mapped by the data map $\Phi$. If $x$ is merely near $\Pf$ then $\Phi(x)$ is not determined \emph{a priori}. However, such an $x$ will strongly exhibit the structure $\Phi$ is meant to detect. So $\Phi$ should recognize that fact and provide a reasonable fit to such a data set. In particular, $x$ should not be a severe singularity of $\Phi$. 

The inequality \eqref{E:Hm.a.S.geq.R.d-p-1}, 
$\Hm^{a}(\Ss') \geq \gamma R^{ \min(d-p-1,a) }$, puts a limit on the degree towhich both desiderata can be achieved. It describes a tradeoff between them. 
\end{remark}

  \begin{remark}[Extended ``sales pitch''] \label{R:extended.sales.pitch}
As noted above Property \ref{Pty:agree.near.T} makes precise the ``sales pitch'', remark \ref{R:sales.pitch}. The theorem suggests an extension. Let $\epsilon > 0$. Suppose one is willing and able to examine, $\epsilon$ units away from $\Pf$,  the behavior of a data map $\Phi$. If one finds that the collection of singularities of $\Phi$ within that distance of $\Pf$ has $\Hm^{a}$ measure 0, then, in theory, one can bound below $\Hm^{a}(\Ss')$. A difficulty is that $\Pf \supset \T$ can be much bigger than $\T$ so verifying \eqref{E:essntl.dist from.S.to.Pf.small} can be much harder than verifying 
the hypotheses of theorem \ref{T:Phi.star.Hr.contains.Theta.star.Hr} 
  \end{remark}

Might $\min(d-p-1,a)$ typically -- always? -- be $d-p-1$? In example \ref{Ex:basic.property.case} and theorem \ref{T:spher.loc.sing.codim.nvar+1} below, for some $r \leq \dim \T$, we have 
$a = d-r-1 \geq d - t - 1$. But $\T \subset \Pf$ so $t \leq p$. Thus, $a \geq d - p - 1$.

Suppose $\min(d-p-1,a) = d-p-1$. Notice that for $R \in (0,1)$ the quantity $R^{d-p-1}$ increases in $p$. Therefore, the larger $p < d$ is the larger is the bound \eqref{E:Hm.a.S.geq.R.d-p-1}. See proposition \ref{P:aug.direct.mean.beats.robst.loc} for an implication of this.

Note that the symbols $\Pf$ and $P$, which differ only by font, have distinct meanings. 
Note further that \eqref{E:essntl.dist from.S.to.Pf.small} holds, 
e.g., if $\dim \bigl\{ x \in \Ss' : dist( x, \Pf ) < R \bigr\} < a$. By \eqref{E:|P|.compact}, we have
    \begin{equation}  \label{E:complex.P.is.finite}
         \text{The simplicial complex } P \text{ is finite.}
    \end{equation}
The discussion of Figure \ref{F:oneZeroDimP} in section \ref{SS:measure.distance.to.P} shows that the exponent $d-p-1$ in \eqref{E:Hm.a.S.geq.R.d-p-1} is not a surprise. 
\eqref{E:aug.mean.sing.vol.asymp}, with $p := \nvar$ and $\Pf := \T$, also provides an example where we get the volume of singular sets dropping off like distance to $\Pf$ raised to the $d-p-1$ power. However, because the constant $\gamma$ in \eqref{E:Hm.a.S.geq.R.d-p-1} is unknown these examples do not prove that the inequality is tight. (However, in the situation sketched in figure \ref{F:oneZeroDimP}(b), by the isoperimetric inequality, Osserman {\cite{rO1978.IsoperimetricInequality}, the constant is clearly $2 \pi$.)

Proposition \ref{P:aug.direct.mean.beats.robst.loc} provides a nontrivial application of theorem \ref{T:lwr.bnd.on.Haus.meas}.

Combining the theorem with \eqref{E:Hm.Rdelta.geq.delta.to.codim.times.Hm.R} and recalling that $\D$ is a compact smooth manifold -- see \eqref{E:D.is.cmpct.Riem.manif} -- (so we can take $\Rcl := \Ss'$ in \eqref{E:Hm.Rdelta.geq.delta.to.codim.times.Hm.R}) we get the following. (See \eqref{E:parallel.body} for definition of $(\Ss')^{\delta}$.)  
   \begin{corly}  \label{C:lwr.bound.on.prob.of.being.close.to.S'}
Let $x$ be a random element of $\D$ whose distribution is absolutely continuous with continuous strictly positive density $h$ w.r.t.\ $\Hm^{d}$. Let $R := dist^{a}(\Ss', \Pf)$.       Under the hypotheses of theorem \ref{T:lwr.bnd.on.Haus.meas}, and proposition \ref{P:tube.about.sing.set}, if $\Hm^{a}(\Ss') < \infty$ (so $\dim \Ss' = a$) there is a constant $C > 0$ depending only on $\D$, $\T$, $C[\Pf]$, $a$, and $\F$ 
      s.t.\ for all sufficiently small $\delta \in (0,1)$ 
	\begin{multline}  \label{E:Prob.delta.R}
	     \text{Prob} \, \bigl\{ x \in (\Ss')^{\delta} \bigr\} 
	       \geq (\inf h) \, \Hm^{d} \bigl[ (\Ss')^{\delta} \bigr]  
	         \geq (\inf h) \, \delta^{d-\dim \Ss'} \, \Hm^{\dim \Ss'} 
	           \bigl[ (\Ss')^{\delta} \bigr]  \\ 
	             \geq (\inf h) \, \delta^{d-a} \, \Hm^{a} 
	               \bigl[ (\Ss')^{\delta} \bigr] 
	             \geq C (\inf h) \, \delta^{d-a} \, R^{\min(d-p-1,a)}.
	  \end{multline}. 
   \end{corly}
(Since $\D$ is compact in theorem \ref{T:lwr.bnd.on.Haus.meas} and $h$ is continuous and strictly positive, $\inf h > 0$.) This result reveals a tension. As we will see in chapter \ref{Chptr:severity}, we often may assume $\Ss'$ consists of ``severe'' singularities of $\Phi$. If that is the case, if $R$ is small we have $\Phi$ falling apart near perfect fits. (See figure \ref{F:CircleSing}(b) and section \ref{SS:measure.distance.to.P}.) This is highly undesirable. Therefore, we want $R$ to be big. On the other hand  we want $\text{Prob} \, \bigl\{ x \in (\Ss')^{\delta} \bigr\}$ to be small. That favors small $R$. However, it may be hard to check that $\Hm^{a}(\Ss') < \infty$.

\section{Proof of theorem 4.2.6} \label{SS:measure.proof}
\subsection{Idea of proof} \label{SSS:idea.of.proof} 
The proof of the theorem is rather long so to orient the reader we give the gist of it here. 
If $\Hm^{a}(\Ss') = \infty$ then \eqref{E:Hm.a.S.geq.R.d-p-1} holds trivially. So assume 
    \begin{equation*}
      \Hm^{a}(\Ss') < \infty .
    \end{equation*} 
In particular, $\dim \Ss' \leq a$. But by property \ref{Pty:agree.near.T}, $\dim \Ss' \geq a$. Thus, $\dim \Ss' = a$. 
We may identify $\D$ and $|P|$. (Remember, $P$ is the simplicial complex. $\Pf$ is the set of perfect fits.) Notice that $R \leq diam(\D)$. 
Suppose $R \geq R_{0} > 0$,  where $R_{0}$ does not depend on $(\Phi, \Ss')$. For simplicty, assume $R_{0} =1$. 

We will observe the following convention in this proof.
    \begin{multline}  \label{E:use.G.invar.version.of.thm}
        \text{Whenever we invoke theorem \ref{T:polyApproxThm} it will be understood that we are using the } \\
           G\text{-invariant version described in proposition \ref{P:Phi.invar.Phi.tilde.invar}. }
    \end{multline}
Perhaps after subdividing $P$ (definition \ref{D:subdivision}), there is a subcomplex 
$Q$ of $P$ s.t.\ 
$|Q| \cap \T = \varnothing$ and $\Hm^{a} \bigl( \Ss' \setminus |Q| \bigr) = 0$. We use theorem \ref{T:polyApproxThm} to approximate $\Phi$ by a $G$-invariant data map, $\tilde{\Phi}$, continuous off a compact set $\tilde{\Ss}$ with the following properties. (i) $\tilde{\Ss} \cap |Q|$ is empty or the underlying space of a subcomplex of $Q$ of dimension no greater than $a$. (ii) Since $\Ss'$ has zero $\Hm^{a}$ measure near $\Pf$, so does $\tilde{\Ss}$. Specifically, $\Hm^{a} \bigl( \tilde{\Ss} \setminus |Q| \bigr) = 0$. 
(iii) There exists $\tilde{K} < \infty$ depending only on $a$ and $Q$ such that
	\begin{equation}  \label{E:polyhedral.volume.mag.factor}
		\Hm^{a} ( \tilde{\Ss} ) = \Hm^{a} \bigl( \tilde{\Ss} \cap |Q| ) 
		  \leq \tilde{K} \Hm^{a} (\Ss').
	\end{equation}
And (iv) $\tilde{\Ss} \cap \T = \Ss' \cap \T$ and the restrictions are equal: 
$\tilde{\Phi} \restriction_{\T \setminus \Ss'} = \Phi \restriction_{\T \setminus \Ss'}$. 

Therefore, by property \ref{Pty:agree.near.T} of $(\Phi, \Ss', G, \T, a)$, we have that 
$a = \dim \Ss' \geq \dim \bigl( \tilde{\Ss} \cap |Q| ) = \dim \tilde{\Ss} \geq a$. (Thus, 
$\dim \bigl( \tilde{\Ss} \cap |Q| \bigr) = a$. But $\tilde{\Ss} \cap |Q$ is the underlying space of a subcomplex of $Q$ so $\dim \bigl( \tilde{\Ss} \cap |Q| \bigr)$ is an integer. Therefore, $a$ is an integer.) That means $\tilde{\Ss}$ includes at least one simplex, $\sigma$, s.t.\ 
$\dim \sigma = a$. Since $P$ is finite, 
$m := \min \bigl\{ \Hm^{a}(\tau) : \tau \in P, \, \dim \tau = a \bigr\} > 0$. Thus, $\Hm^{a} \bigl( \tilde{\Ss} \cap |Q| ) \geq m > 0$. Therefore, by \eqref{E:polyhedral.volume.mag.factor}, $\Hm^{a} (\Ss') \geq \tilde{K}^{-1} \Hm^{a} ( \tilde{\Ss} ) \geq m/\tilde{K}$. 
Thus, if $R \in [1, diam(\D)]$, \eqref{E:Hm.a.S.geq.R.d-p-1} holds 
with 
    \begin{equation*}
      \gamma = \gamma_{1} := m / \bigl[ \tilde{K} diam(\D)^{\min(d-p-1, a)} \bigr] .
    \end{equation*}

Suppose $R < 1$. For simplicity, imagine $\Pf = \RR^{p} \subset \RR^{d}$ and, for $y \in \Pf$, the fiber, $C[y]$, is just the product $\{y\} \times B^{d-p}_{r}$, where $r > 1$ and $B^{d-p}_{r}$ is the open ball of radius $r$ in $\Pf^{\perp} \cong \RR^{d-p}$, the  orthogonal complement of $\Pf$. (See \eqref{E:ball.defn} and \eqref{E:superscript.perp.notation}.) 
Then $\Hm^{a} \bigl[ \Pf \times B_{R}^{d-p}(0) \bigr] = +\infty$. 

In this context, if $y \in \Pf$ and $v \in \Pf^{\perp}$ then $Exp(y,v) = y+v$ and the inverse, $\alpha$, is orthogonal projection onto $\Pf$. As in definition \ref{D:fibering.by.cones} part \eqref{I:Exp.alpha.homeom}, let 
$\mcl{C} := Exp \bigl( C[\Pf] \bigr)$. We apply a piecewise linear operation to pull 
$\alpha(\Ss' \cap \mcl{C})$ away from $\Pf \times \{0\}$ along the fibers of $C[\Pf]$. The effect, after exponentiating, is to pull $\Ss'$ away from $\Pf$. 
And we make the corresponding changes in $\Phi$. We call this operation ``dilation'' (section \ref{SS:Dilation}). This gives rise to a new $\Phi$, called $\Phi_{dilate}$ with its singular set $\Ss_{dilate}$. Perform the dilation operation to a degree sufficient to make 
$dist^{a}(\Ss_{dilate}, \Pf) \geq R_{0} = 1$. Then we know from before 
that $\gamma_{1} \leq \Hm^{a}(\Ss_{dilate}')$. Suppose we can find a factor 
$F < \infty$ s.t.\ 
$\gamma_{1} \leq \Hm^{a} (\Ss_{dilate}') \leq F \Hm^{a} (\Ss')$.
Then $\Hm^{a} (\Ss') \geq \gamma_{1}/F$. 

The dilation operation turns out to be Lipschitz with Lipschitz constant of the form  
to $K_{dil}/R$. Therefore, we can apply \eqref{E:Lip.magnification.of.Hm} and find that we may take $F = K_{dil}^{a} R^{-a}$. 
 
If $a \leq d-p-1$ then $F \propto R^{-a}$ can be the best one can do.
To see this, suppose $\dim \Ss' = a \leq d-p-1$. Recall that $a \geq 0$ is an integer. Although it might not correspond to any actual $\Phi$, one way we could have 
$\dim \Ss' \leq d-p-1$ is if there were and an $a$-sphere $S$ with radius $R$ and centered at the origin s.t.\ $\{y\} \times S \subset \Ss'$ for some $y \in \Pf$. (Do not confuse $S$ 
with $\Ss'$.) Dilation expands $\{y\} \times S \subset \Ss'$ in all directions equally so dilation expands $\Hm^{a}(S)$ by a factor proportional to $R^{-a}$. (Remember $R =$ diameter of $S$, which is expanded during dilation.) This yields \eqref{E:Hm.a.S.geq.R.d-p-1} in this special case.

But what if $d-p-1 < a$? In that case, if $F \propto R^{-a}$ we wind up with \eqref{E:Hm.a.S.geq.R.d-p-1} but with the wrong power of $R$. If $a > d-p-1$, 
then an $a$-sphere cannot fit into $\Pf^{\perp}$. But we might have instead 
$\Ss' = \mcl{E} \times S$, where $\mcl{E}$ is an at least 
$\bigl[ a - (d-p-1) \bigr]$-dimensional subset of $\Pf$ and $S$ is the $(d-p-1)$-sphere 
$S \subset \RR^{d-p} = \Pf^{\perp}$ with radius $R$ and centered at the origin. (This makes 
$\dim \Ss' \geq \bigl[ a - (d-p-1) \bigr] + (d-p-1 = a$, consistent with property \ref{Pty:agree.near.T}.) During dilation $\mcl{E}$ does not expand. $\Ss'$ only expands, and in $d-p-1$ directions at once. Thus, $\Hm^{a}(\Ss')$ expands only by a factor, 
$F = K_{dil}^{d-p-1} R^{-(d-p-1)}$ giving \eqref{E:Hm.a.S.geq.R.d-p-1} in this example with 
    \begin{equation*}
      \gamma = \gamma_{1} K_{dil}^{-(d-p-1)} . 
    \end{equation*}
We may assume $K_{dil} \geq 1$ in which case we get a $\gamma$ that works whether or not $a \leq d-p-1$. 

To prove \eqref{E:Hm.a.S.geq.R.d-p-1} in general, the bound \eqref{E:Lip.magnification.of.Hm} is too crude, as we just saw. Instead, we compute $F$ using the ``area formula''.  

Actually, one must be careful here. The area formula requires that $\Ss'$ be ``countably rectifiable'', which a general closed set $\Ss'$ might not be (Federer \cite[3.3.19, pp.\ 302--306]{hF69}). However, a polyhedron (appendix \ref{Chptr:basics.of.simp.comps}) \emph{is} countably rectifiable. So once again we use theorem \ref{T:polyApproxThm} to replace $\Ss'$ by a set that is $\Hm^{a}$-almost everywhere the underlying space of a subcomplex of $P$ (subsection \ref{SS:Small.R.rect}). (Subdividing the complex $P$ may be necessary first.)

Moreover, in general $\Pf$ can, alas, be a stratified space with curved strata. The same is true of the ``links'' $\msf{L}_{i}$ that define the cone fibers (definition \ref{D:fibering.by.cones}, part \ref{I:local.triv}). On top of that there are reasons to consider a subset of $C[\Pf]$ in which the fibers are tapered to 0 off a compact neighborhood 
of $\T$. This complicates the calculations (a lot!), but the same idea applies. We proceed in stages by change of variables and approximation until we reduce the problem to analyzing the behavior of a vector-valued rational function on a bounded convex set.

\subsection{A neighborhood of $\T$}   \label{SS:neighborhood.of.T} 
By assumption, $\T \subset \D$ is a compact $t$-dimensional manifold. If $\T$ is an open subset of $\Pf$ (e.g., $\Pf = \T$ or $\dim \Pf = 0$; see example \ref{E:cones.over.0.dim.Pf}) define $\mcl{W} := \clU := \T$. 

Suppose $\T$ is not open in $\Pf$. If $\mcl{A} \subset X \subset \D$, 
let $clos_{\X}(\mcl{A})$ denote closure of $\mcl{A}$ relative to $\X$. 
With $\X = \Pf$, we have $clos_{\Pf}(\mcl{A}) = clos_{\D}(\mcl{A}) \cap \Pf$:
    \begin{multline}  \label{E:Pf.and.D.closures}
      clos_{\Pf}(\mcl{A}) 
        = \bigcap_{\mcl{A} \subset \mcl{C} \subset \Pf; \; \mcl{C} \text{ is closed in } \Pf} \mcl{C}
        = \bigcap_{\mcl{A} \subset \mcl{F}; \; \mcl{F} \text{ is closed in } \D} (\mcl{F} \cap \Pf) \\
        = \left( \bigcap_{\mcl{A} \subset \mcl{F}; \; \mcl{F} \text{ is closed in } \D} \mcl{F} \right) 
           \cap \Pf
       = clos_{\D} (\mcl{A}) \cap \Pf. 
    \end{multline}

By \eqref{E:Pf.loc.compct.strat.space}, $\Pf$ is locally compact. Since $\T$ is compact it follows that there exists an open neighborhood, 
$\clU_{1}' \subset \Pf$ of $\T$ s.t.\ $clos_{\Pf}(\clU_{1}')$ is compact. Choose positive $\delta$ less than 
$\text{dist}(\Pf \setminus \clU_{1}', \T) := \inf \bigl\{ \xi(x,y) : 
x \in \Pf \setminus \clU_{1}', \, y \in \T \bigr\} > 0$. (See \eqref{E:set.distances}.)
Let $\clU_{\D} := \bigl\{ x \in \D : \text{dist}(x, \T) < 2\delta/3 \bigr\}$. 
So in general $\clU_{\D} \nsubseteq \Pf$ and, making $\delta$ smaller if necessary, 
$clos_{\D}(\clU_{\D})$ is compact. Now, $\clU_{\D} \cap \Pf \subset \clU_{1}'$. 
Therefore, 
    \begin{equation}  \label{E:closure.of.U.D.is.compact}
      clos_{\D}(\clU_{\D}) \text{ and } 
        clos_{\Pf}(\clU_{\D} \cap \Pf) \subset clos_{\Pf}(\clU_{1}') \text{ are compact.}
    \end{equation} 
Since both $\xi$ and $\T$ are $G$-invariant, we have, that $\clU_{\D}$ is $G$-invariant. 
Since $\D$ is metric, it is normal (Simmons \cite[p.\ 133]{gfS63}) and since $\T$ compact, 
$\T$ has a neighborhood, $\mcl{W}_{\D}$, in $\D$ s.t.\ 
$clos_{\D}(\mcl{W}_{\D}) \subset \clU_{\D}$. Therefore, letting
$\mcl{W} := \mcl{W}_{\D} \cap \Pf \subset \clU := \clU_{\D} \cap \Pf \subset \clU_{1}'$, we have
	\begin{equation}  \label{E:W.U.T.subset}
	     clos_{\Pf}(\clU) \text{ is compact and }
		\T \subset \mcl{W} \subset clos_{\Pf}(\mcl{W}) \subset \clU
		         \subset clos_{\Pf}(\clU) \subset \Pf.
	\end{equation}
Replacing $\mcl{W}_{\D}$ by $\bigcap_{g \in G} g(\mcl{W}_{\D})$, we have, by \eqref{E:g.commutes.w/.set.ops}, that $\mcl{W}_{\D}$ is $G$-invariant, too. 
Moreover, $\Pf$ is also $G$ invariant, by assumption. It follows that $\clU$ and $\mcl{W}$ are $G$-invariant. 

Since $clos_{\Pf}(\clU)$ is compact it is closed in $\D$. Therefore, 
$clos_{\D}(\clU) \subset clos_{\Pf}(\clU) \subset \Pf$. Thus, by \eqref{E:Pf.and.D.closures}, $clos_{\Pf}(\clU) = clos_{\D}(\clU)$. Define 
    \begin{equation}  \label{E:overline.U.defn}
        \overline{\clU} := clos_{\Pf}(\clU) = clos_{\D}(\clU). \; 
          \overline{\clU} \text{ is compact.}
    \end{equation}
In fact, from now on define 
    \begin{equation*}
      \overline{\X} := clos_{\D}(\X), \text{ for any } \X \subset \D .
    \end{equation*}

Since $\overline{\clU}$ is compact, there exists $\epsilon > 0$, constant s.t.\ 
$0 < \epsilon < \epsilon_{\Pf}(x')$ for every 
$x' \in \overline{\clU}$. ($\epsilon_{\Pf}$ is defined following lemma \ref{L:2eps.P.exists}.) 

\emph{Claim:} We may take  
	\begin{equation}  \label{E:eps.=.2}
		\epsilon = 2 . \text{ Thus, } \mbf{F}_{[0,4]}[\overline{\clU}] 
		  \subset C[\overline{\clU}] .
	\end{equation}
(See \eqref{E:boldF.[E].defns}.) This is accomplished by replacing 
$\langle \cdot, \cdot \rangle$ by 
$\langle \cdot, \cdot \rangle_{new} := 4 \epsilon^{-2} \langle \cdot, \cdot \rangle$. 
Obviously, if $(y,v) \in \mbf{F}_{[0,2 \epsilon]}[\overline{\clU}]$ w.r.t.\ $\langle \cdot, \cdot \rangle$, then $(y,v)$ is in $\mbf{F}_{[0,4]}[\overline{\clU}]$ w.r.t.\ 
$\langle \cdot, \cdot \rangle_{new}$, and \emph{vice versa}. 

If $\nabla$ is the Riemannian connection on $\D$ w.r.t.\ $\langle \cdot, \cdot \rangle$ then, by Boothby \cite[Definition (3.1), p.\ 313]{wmB75}, $\nabla$ is also the Riemannian connection on $\D$ w.r.t.\ $\langle \cdot, \cdot \rangle_{new}$. Therefore, geodesics are the same for the two Riemannian metrics (Boothby \cite[Definition (5.1), p.\ 326]{wmB75}). Moreover, the lengths of a geodesic w.r.t.\ $\langle \cdot, \cdot \rangle_{new}$ is $2/\epsilon$ times that for $\langle \cdot, \cdot \rangle$ (Boothby \cite[p.\ 185]{wmB75}). Hence, recalling \eqref{E:essential.dist.defn}, we have that $dist^{a}$ w.r.t.\ 
$\langle \cdot, \cdot \rangle_{new}$, call it $dist^{a,new}$, is $2/\epsilon$ times that 
for $\langle \cdot, \cdot \rangle$. 

Hence, if $dist^{a, new}(\Ss', \Pf) \geq R_{new}$, then 
$dist^{a}(\Ss', \Pf) \geq \epsilon R_{new}/2$. 
Similarly, by \eqref{E:Hs.defn}, Hausdorff measure, $\Hm^{a}$ w.r.t.\ 
$\langle \cdot, \cdot \rangle_{new}$, call it 
$\Hm_{new}^{a}$, is $(2/\epsilon)^{a}$ times that w.r.t.\ $\langle \cdot, \cdot \rangle$. 
Therefore, expressed in the ``new'' framework, \eqref{E:Hm.a.S.geq.R.d-p-1} looks like this:
             \begin{equation*}  
                (\epsilon/2)^{a} \Hm_{new}^{a}(\Ss') = \Hm^{a}(\Ss') 
                  \geq \gamma \; (\epsilon R_{new}/2)^{ \min(d-p-1,a) } .
             \end{equation*}
This can be written:
             \begin{equation*}  
                \Hm_{new}^{a}(\Ss') 
                  \geq \bigl[ (\epsilon/2)^{ \min(d-p-1,a) - a } \, \gamma \bigr] 
                    R_{new}^{ \min(d-p-1,a) }.
             \end{equation*}
Hence, replacing $\langle \cdot, \cdot \rangle$ by $\langle \cdot, \cdot \rangle_{new}$ does not qualitatively change \eqref{E:Hm.a.S.geq.R.d-p-1}. Only an adjustment in $\gamma$ is needed. This proves the claim \eqref{E:eps.=.2}.

Since $0 < \epsilon < \epsilon_{\Pf}$ on $\clU$, we have
	\begin{equation}  \label{E:eps.P.geq.2.on.U}
		\epsilon_{\Pf} \restriction_{\clU} > 2.
	\end{equation}

By Spivak \cite[Theorem 15, p.\ 68]{mS79.SpivakVol1} using the open cover of $\D$, 
$\bigl\{ \clU_{\D} , clos_{\D}(\mcl{W}_{\D})^{c} \bigr\}$ (complement taken relative 
to $\D$), there exists
	\begin{equation}  \label{E:rho.in.[0,1]}
		\rho : \D \to [0,1]
	\end{equation}
that is smooth and satisfies
	\begin{multline}  \label{E:rho.on.W.off.U}
		\rho \equiv 1 \text{ on } clos_{\D}(\mcl{W}_{\D}) \text{ and } 
		          \rho \equiv 0 \text{ everywhere on } \D \setminus \clU_{\D}. \\
		   \text{ In particular, the restriction } \rho \restriction_{\Pf} = 0 
		     \text{ off } \clU. 
	\end{multline}
(If $\T$ is open in $\Pf$, just take $\rho$ to be the indicator or characteristic function (see \eqref{E:indicator.fn.defn}) of $\T$.) In the interest of brevity, in a context dependent fashion we will often use $\rho$ to denote the restriction $\rho \restriction_{\Pf}$ of $\rho$ to $\Pf$. 

By \eqref{E:closure.of.U.D.is.compact}, $\clU_{\D}$ has a relatively compact neighborhood. 
$\mcl{V}_{\D}$. By corollary \ref{C:cont.diff.=.loc.Lip}, $\rho$ is locally Lipschitz on $\D$. Recall \eqref{E:D.imbedded.in.Rk}. Hence, by \eqref{E:closure.of.U.D.is.compact} and lemma \ref{L:imbedding.is.loc.Lip}, it is Lipschitz on $\mcl{V}_{\D}$ w.r.t.\ both $\xi$ and the Euclidean distance on $\RR^{k}$. Since $\rho$ is 0 off $\clU_{\D}$ it follows that
    \begin{equation}  \label{E:rho.is.Lip.on.D}
      \rho \text{ is Lipschitz on } \D \text{ w.r.t.\ both $\xi$ and the Euclidean distance on } 
        \RR^{k} . 
    \end{equation}

By \eqref{E:eps.=.2} and the fact that 
if $y \in \Pf \setminus \clU$ then $\bigl( y, \rho(y) v \bigr) = (y,0) \in \Pf$, we have, 
	\begin{equation*}
		\bigl( y, 2 \rho(y) v \bigr) \in C[\Pf] \text{ for every } (y,v) \in \mbf{F}_{1}[\Pf]. 
	\end{equation*}

Because $\clU_{\D}$, and $\mcl{W}_{\D}$ are both $G$-invariant, we may assume
	\begin{equation}  \label{E:rho.0.off.U.rho.g-invar}
		      \rho \circ g = \rho \text{ for every } g \in G.
	\end{equation}
(If $\rho$ is not initially $G$-invariant, replace it by $|G|^{-1} \sum_{g \in G} \rho \circ g$, where $|G|$ is the cardinality of $G$.) 

We may assume
	\begin{equation}  \label{E:rho.strictly.poz.on.U}
		\rho \text{ is strictly positive on } \clU_{\D}.
	\end{equation}
If this is initially false, then just replace $\clU_{\D}$ 
by $\{ x \in \D : \rho(x) > 0 \} \subset \clU_{\D}$. \eqref{E:rho.on.W.off.U} and \eqref{E:rho.0.off.U.rho.g-invar} continue to hold.

Let 
	\begin{equation}  \label{E:Cs.defn}
	      C_{s} := C(s) 
	           := \bigl\{ (x', v) \in C[\Pf]  : |v| < \rho(x') s \bigr\}, \quad s \in [0,\twoess].
	\end{equation}
Thus, if $\dim \Pf = 0$, so we adopt the framework of example \ref{E:cones.over.0.dim.Pf}, we see that $C_{s}$ is just the union $\bigcup_{y \in \T} \; \{y\} \times B_{s}(0)$, where 
$B_{s}(0)$ is the ball about 0 in $\RR^{d}$ with radius $s$. (See \eqref{E:ball.defn}.) Recall \eqref{E:pi.is.proj}. By \eqref{E:rho.on.W.off.U}, 
    \begin{equation*}
      \pi(C_{s}) \subset \clU .
    \end{equation*}

Let $(x,v) \in T \D$. By \eqref{E:pi.is.proj}, $\pi(x,v) = x$. Let $\varkappa$ be the other projection: 
    \begin{equation}  \label{E:varkappa.is.other.proj}
      \varkappa(x,v) = v , \quad (x,v) \in T \D .
    \end{equation} 
By \eqref{E:pi.is.smooth.and.open}, $\pi$ is continuous. Now, by \eqref{E:omega.D.defn}, \eqref{E:xi+.from.2.metrics}, and  \eqref{E:omegaD.xi+.Lip.property}, 
$T \D \subset \D \times \RR^{k} $ has the relative topology it inherits from the product topology on $\D \times \RR^{k}$. Therefore, $\varkappa$ is also continuous. In fact, 
$\varkappa$ is clearly Lipschitz relative to $\xi_{+}$ and the Euclidean metric 
on $\RR^{k}$.

Since $\varkappa$, $\rho$, and  $| \cdot | : (x,v) \mapsto |v|$ are continuous, we see that 
for any $s > 0$, the set 
    \begin{equation}  \label{E:Cs.relatively.open}
      C_{s} \subset C[\Pf] \text{ is open in the relative topology.}  
    \end{equation}
So, by \eqref{E:rho.in.[0,1]}, \eqref{E:boldF.[E].defns}, \eqref{E:W.U.T.subset}, \eqref{E:rho.on.W.off.U}, and \eqref{E:eps.=.2}, 
	\begin{equation}  \label{E:C.sub.s.in.Bold.F.[0,s)}
	  C_{s} \subset \mbf{F}_{[0,s)}[\clU]  \text{ and }
		\T \subset C_{s} \subset \mbf{F}_{[0,s]}[\clU]
		  \subset \mbf{F}_{[0,\twoess]}[\clU] \subset C[\clU],
		    \quad s \in [0,\twoess] .
	\end{equation}
 
Recall definition \ref{D:fibering.by.cones} part \eqref{I:Exp.alpha.homeom}. 
For $s \in (0, \twoess]$ let 
	\begin{equation}   \label{E:Bs.defn}
		\mcl{B}_{s} := Exp (C_{s}) = \alpha^{-1} ( C_{s} ) \subset \mcl{C} \subset \D.
	\end{equation}
Therefore, by \eqref{E:Cs.relatively.open} and definition \ref{D:fibering.by.cones}, part \ref{I:Exp.alpha.homeom}, we have 
	\begin{equation}  \label{E:Bs.contains.nbhd.of.T}
           \mcl{B}_{s} \text{ is an open neighborhood of } \T.
	\end{equation}
By \eqref{E:C.C[P].G.invar}, \eqref{E:rho.0.off.U.rho.g-invar}, \eqref{E:gExp=Expg.star}, and the fact that $G$ is a group of isometries, we have
	\begin{equation}  \label{E:G.invariance.of.BS.CS}
		g_{\ast}(C_{s}) = C_{s} \text{ and } g(\mcl{B}_{s}) = \mcl{B}_{s}, 
		  \quad g \in G, s \in (0, \twoess].
	\end{equation}
By \eqref{E:C.sub.s.in.Bold.F.[0,s)} (and definition \ref{D:fibering.by.cones} part \eqref{I:Exp.alpha.homeom}), we have 
	\begin{equation} \label{E:B2.subset.C}
		\mcl{B}_{\twoess} \subset \overline{\mcl{B}}_{\twoess} \subset \mcl{C}.
	\end{equation}

Recall from definition \ref{D:fibering.by.cones} that $\pi_{C} : C[\Pf] \to \Pf$ is projection onto the first factor and 
$\alpha := \bigl( Exp \restriction_{C[\Pf]} \bigr)^{-1}: \mcl{C} \to C[\Pf]$. Since $0 \leq \rho \leq 1$, by \eqref{E:rho.in.[0,1]}, by \eqref{E:dist.Exp.v.to.P.bnd}, we have 
	\begin{equation}  \label{E:Bs.w/in.s.of.P}
		x \in \mcl{B}_{s} \text{ implies } dist(x, \Pf) 
		  < \rho \bigl[\pi_{C} \circ \alpha(x) \bigr] s \leq s, 
		    \quad (s \in (0, \twoess]),
	\end{equation} 
where $dist$ is calculated using $\xi$. Let $y$ be a boundary point of $\mcl{B}_{s}$. Then letting $x \to y$ through $\mcl{B}_{s}$ we get
    \begin{equation} \label{E:dist.outside.of.Bs.to.P}
        dist(\D \setminus \mcl{B}_{s}, \Pf) 
          \leq \lim_{x \to y; \, x \in \mcl{B}_{s}} dist(x, \Pf) \leq s.
    \end{equation}

Similarly, let $s \in (0,1)$, $y \in \T$, and let $v \in \RR^{k}$ satisfy $|v| \in (s,2)$. Then, by \eqref{E:Cs.defn}, we have $(y,v) \in T \D \restriction_{\Pf} \setminus C_{s}$. 
Let $x := Exp(y,v)$. Then, because $y \in \T$, by \eqref{E:dist.Exp.v.to.P.bnd}, 
$dist(x, \T) \leq |v|$. Letting $|v| \downarrow s$ it follows that 
    \begin{equation} \label{E:Bs.w/in.s.of.T}
        dist(\D \setminus \mcl{B}_{s}, \T) \leq \lim_{|v| \downarrow s} dist(x, \T) = s.
    \end{equation}

By \eqref{E:C.sub.s.in.Bold.F.[0,s)} and part \ref{I:Exp.alpha.homeom} of definition \ref{D:fibering.by.cones}, we have
	\begin{equation}   \label{E:alpha.Bs.=.Cs}
		\alpha(\mcl{B}_{s}) = C_{s}, \quad s \in (0,\twoess].
	\end{equation}

\subsection{\eqref{E:Hm.a.S.geq.R.d-p-1} with large $R$}  \label{SS:R.geq.R0} 
Let $R_{0} \in (0,\ess)$ be a constant satisfying 
	\begin{equation} \label{E:outside.B.toT.>.R0}
		dist(\D \setminus \mcl{B}_{1}, \T) > R_{0}. 
	\end{equation}
(Note it is the distance from $\D \setminus \mcl{B}_{1}$ to $\T$, not $\Pf$, that is being bounded here. See remark \ref{R:local.level.theorem}.) For example,
 	\begin{equation}  \label{E:R0.choice}
	   \gamma \leq \gamma(R_{0}) , \text{ where, for example, }
	     R_{0} := \tfrac{1}{2} \min \bigl\{ \ess, dist( \D \setminus \mcl{B}_{1}, \T) \bigr\} .
	\end{equation}
which depends only on $\D$, $\T$, $C[\Pf]$, $a$, and $\F$. (By compactness of $\T$, $R_{0} > 0$.)  But sometimes we allow $R_{0}$ to depend on data map(s). By \eqref{E:Bs.contains.nbhd.of.T}, $\mcl{B}_{1}$ is an open neighborhood of $\T$ and $\T$ is compact by assumption so the distance is positive. Moreover, it is ordinary ``$dist$'' that is intended here (see \eqref{E:set.distances}), not essential distance \eqref{E:essential.dist.defn}.)
Suppose 
	\begin{equation}   \label{E:Hma.S'.cap.B.eta.R0.0}
		\Hm^{a} (\Ss' \cap \mcl{B}_{R_{0}}) = 0.
	  \end{equation}
Note that, by \eqref{E:Bs.w/in.s.of.P}, the following holds.
    \begin{equation} \label{E:R.R0.BR0}
        \text{If \eqref{E:essntl.dist from.S.to.Pf.small} holds for } R \geq R_{0}, 
          \text{ then \eqref{E:Hma.S'.cap.B.eta.R0.0} holds.}
    \end{equation}

By assumption and \eqref{E:complex.P.is.finite}, $\D$ has a $G$-invariant bi-Lipschitz triangulation $f : |P| \to \D$, where 
$P$ is a finite simplicial complex and $\{ f^{-1} \circ g \circ f : |P| \to |P| : g \in G \}$ is a finite group of simplicial homeomorphisms on $P$. By \eqref{E:Hma.S'.cap.B.eta.R0.0},  lemma \ref{L:loc.Lip.image.of.null.set.is.null}, and \eqref{E:g.commutes.w/.set.ops}, we have 
    \begin{equation}  \label{E:Hma(finvers(S').cap.finvers(BR0))=0}
      \Hm^{a} \bigl[ f^{-1}(\Ss') \cap f^{-1}(\mcl{B}_{R_{0}}) \bigl] = 0 .
    \end{equation}  

Temporarily identify $\Phi$ with $\Phi \circ f$, $\D$ with $|P|$, $\Pf$ with $f^{-1}(\Pf)$, $\Ss'$ with $f^{-1}(\Ss')$, $\mcl{B}_{R_{0}}$ with $f^{-1}(\mcl{B}_{R_{0}})$, and $\gamma$ with $\gamma'$. We may also assume 
    \begin{equation*}
      G \text{ is a finite group of simplicial homeomorphisms on } |P| .
    \end{equation*} 
It remains true that $g(\mcl{B}_{R_{0}}) = \mcl{B}_{R_{0}}$ as in \eqref{E:G.invariance.of.BS.CS}.

\emph{Claim:} We may assume 
	\begin{multline}  \label{E:Q.stays.away.from.T}
		P \text{ has a subcomplex, } Q, \text{ s.t.\ }
		( \D \setminus \mcl{B}_{R_{0}} ) \subset |Q| \\
		 \text{ and } |Q| \text{ does not intersect any simplex } \\
		 \text{ that in turn intersects a simplex intersecting } \T. 
	\end{multline}
To show this, recall \eqref{E:diam.of.set}. By proposition \ref{P:Phi.invar.Phi.tilde.invar}, we may assume $G$ is a group of simplicial is homeomorphisms of $P$ onto itself and  
	\begin{equation}  \label{E:diam.simps.<.R0/3}
		\text{All simplices in } P \text{ have diameter } < R_{0}/3.
	\end{equation} 

Let 
    \begin{multline}  \label{E:Q.defn}
      Q \text{ be the subcomplex of } P 
        \text{ consisting of all simplices that intersect } \\
          \D \setminus \mcl{B}_{R_{0}} \text{ and all faces of such simplices.}
    \end{multline} 
Let $\sigma \in P$ have nonempty intersection with $|Q|$. Then there exists $\tau \in Q$ s.t.\ 
$\tau \setminus \mcl{B}_{R_{0}} \neq \varnothing$ and $\sigma \cap \tau \ne \varnothing$. Suppose $\zeta \in P$ has nonempty intersection with $\T$ and with $\sigma$. 
Let $x_{1} \in \tau \setminus \mcl{B}_{R_{0}}$, $x_{2} \in  \sigma \cap \tau$, 
$x_{3} \in \sigma \cap \zeta$, and $x_{4} \in \zeta \cap \T$. Then, by \eqref{E:outside.B.toT.>.R0} and identifying $x_{i}$ with $f^{-1}(x_{i})$, we have
	\[
		R_{0} < \xi(x_{1}, x_{4}) 
		  \leq \xi(x_{1}, x_{2}) + \xi(x_{2}, x_{3}) + \xi(x_{3}, x_{4}). 
	\]
Now, $x_{1}, x_{2} \in \tau$, $x_{2}, x_{3} \in \sigma$, and $x_{3}, x_{4} \in \zeta$. Therefore, by \eqref{E:diam.simps.<.R0/3} 
we have $R_{0} < 3 \cdot \tfrac{1}{3} R_{0} = R_{0}$. 
Contradiction. Thus, one of the intersections, $\sigma \cap |Q|$, $\zeta \cap \T$, or $\sigma \cap \zeta$ must be empty. This proves the claim \eqref{E:Q.stays.away.from.T}. As a corollary we have,
	\begin{equation}  \label{E:no.simp.touch'n.T.is.face.of.simp.touch'n.Q}
		\text{No simplex intersecting $\T$ is a face of any simplex intersecting } |Q|.
	\end{equation}

Recall definition \ref{D:simplicial.homeom}. Now, $G$ is a finite group of simplicial homeomorphisms on $|P|$. \emph{Claim:} For every $g \in G$ the restriction 
$g \restriction_{|Q|}$ is a simplicial homeomorphism of $Q$ onto itself. Since $g$ is a simplicial homeomorphism of $P$ onto itself, we have that $g : |Q| \to |P|$ is simplicial (definition \ref{D:simplicial.map}). Let $\sigma \in Q$. There exists $\tau \in P$ s.t.\ 
$\sigma$ is a face of $\tau$ and $\tau$ intersects $\D \setminus \mcl{B}_{R_{0}}$. Then $g(\sigma)$ is a face of $g(\tau)$ and, by \eqref{E:G.invariance.of.BS.CS} and \eqref{E:g.commutes.w/.set.ops}, 
$g(\tau) \cap (\D \setminus \mcl{B}_{R_{0}}) = g \bigl( \tau \cap \D 
\setminus \mcl{B}_{R_{0}} \bigr)\neq \varnothing$. Hence $g(\sigma) \in Q$. Since the same is true with $g^{-1}$ in place of $g$, the claim follows.

By property \ref{Pty:agree.near.T} (and \eqref{E:Lip.magnification.of.Hm}), 
    \begin{equation} \label{E:Hma(S')>0}
      \Hm^{a}(\Ss') > 0 \text{ so } \dim \Ss' \geq a .
    \end{equation} 
If $\Hm^{a}(\Ss') = + \infty$ then
\eqref{E:Hm.a.S.geq.R.d-p-1} holds trivially. So we may assume
	\begin{equation}  \label{E:Hm.a.S.finite}
		\Hm^{a}(\Ss') < \infty.
	\end{equation}
But by the definition, \eqref{E:Haus.dim.defn},  of Hausdorff dimension, 
if $\dim \Ss' > a$ then $\Hm^{a}(\Ss') = + \infty$. In summary, we may assume
	\begin{equation}  \label{E:dim.S'.=.a}
		\Hm^{a}(\Ss') < \infty \text{ and } \dim \Ss' = a.
	\end{equation}

By \eqref{E:dim.S'.=.a}, specifically $\dim \Ss' = a$, we may apply proposition \ref{P:Phi.invar.Phi.tilde.invar} to infer the existence of a $G$-invariant continuous map 
$\tilde{\Phi} : |P| \setminus \tilde{\Ss} \to \F$, related to $(\Phi, \Ss')$ as described in theorem \ref{T:polyApproxThm}. In particular, $\tilde{\Ss}$ is closed and $G$-invariant. Moreover, by part \ref{I:S.tilde.subcomp} of theorem \ref{T:polyApproxThm}, 
$\tilde{\Ss} \cap |Q|$ is empty or the underlying space of a subcomplex of $Q$ of dimension no greater than $a$.

\emph{Claim:} 
    \begin{equation}  \label{E:S.tilde.has.empty.intrr}
      \tilde{\Ss} \text{ has empty interior.}
    \end{equation} 
We have been blurring the distinction between $\D$ and $P$, but $\tilde{\Ss}$ has been constructed in the $P$-world. By \eqref{E:Haus.dim.of.cmplx}, 
$\dim \tilde{\Ss} \leq a$. Since the triangulation 
$f : |P| \to \D$ is Lipschitz, by lemma \ref{L:loc.Lip.image.of.null.set.is.null}, 
$\dim f(\tilde{\Ss}) \leq a$. Since $f$ is a homeomorphism (see \eqref{E:triangulation.defn}) $\tilde{\Ss}$ non-empty interior if and only if 
$f(\tilde{\Ss})$ does. Suppose $\tilde{\Ss}$ has non-empty interior. Then there exists an open set $H \subset f(\tilde{\Ss})$. By Boothby \cite[Example (1.6), p.\ 56]{wmB75}, $H$ is a $d$-manifold. Therefore, by \eqref{E:Haus.dim.s-manif.=.s}, 
$a \geq \dim f(\tilde{\Ss}) \geq \dim H = d > a$, a contradiction proving the claim \eqref{E:S.tilde.has.empty.intrr}. 

Now by \eqref{E:Q.stays.away.from.T}, 
if $\tau \in P$, $\tau \cap \T \neq \varnothing$, and $\rho \in Q$,  
then $\tau \cap \rho = \varnothing$. Therefore, tivially,  
by part \ref{I:no.change.off.nbhd.of.S.in.Q} of theorem \ref{T:polyApproxThm}, 
$\tilde{\Ss} \cap \tau = \Ss' \cap \tau$ 
and $\tilde{\Phi} \restriction_{\tau} = \Phi \restriction_{\tau}$. It follow that 
$\tilde{\Ss} \cap \T = \Ss' \cap \T$ and $\tilde{\Phi} \restriction_{\T \setminus \tilde{\Ss}} 
= \Phi \restriction_{\T \setminus \Ss'}$. 
By assumption, $(\Phi, \Ss', G, \T, a)$ satisfies property \ref{Pty:agree.near.T}.  
Now, $(\tilde{\Phi}, \tilde{\Ss})$ are $G$-invariant and $\tilde{\Ss}$ has empty interior. Therefore, by property \ref{Pty:agree.near.T} of $(\Phi, \Ss', G, \T, a)$, 
    \begin{equation}  \label{E:Hma(S.tilde)>0}
      \Hm^{a} (\tilde{\Ss}) > 0. \text{ In particular, } \dim \tilde{\Ss} \geq a .
    \end{equation}
 
Let $y \in |P| \setminus |Q|$ and let $\rho \in P$ be the unique simplex s.t.\ 
$y \in \text{Int} \, \rho$. (See \eqref{E:x.in.exctly.1.simplex.intrr}.) Then $\rho \notin Q$. Hence, $(\text{Int} \, \rho) \cap |Q| = \varnothing$. By \eqref{E:Q.stays.away.from.T}, this means $\text{Int} \, \rho \subset \mcl{B}_{R_{0}}$. Hence, 
$\Hm^{a} \bigl[ \Ss' \cap (\text{Int} \, \rho) \bigr] = 0$. It also follows that, 
    \begin{equation}  \label{E:|P|.less.|Q|.subset.BR0}
      |P| \setminus |Q| \subset \bigcup_{(\text{Int} \, \rho) \cap |Q| = \varnothing} 
        \text{Int} \, \rho \; \subset \; \mcl{B}_{R_{0}} .
    \end{equation}
Therefore, 
    \begin{equation}  \label{E:S.tilde-|Q|.exprssn}
       \tilde{\Ss} \setminus |Q| \subset \bigcup_{(\text{Int} \, \rho) \cap |Q| = \varnothing} 
        \tilde{\Ss} \cap (\text{Int} \, \rho) 
          \; \subset \; \tilde{\Ss} \cap \mcl{B}_{R_{0}} . 
    \end{equation} 
Now, by \eqref{E:Hma.S'.cap.B.eta.R0.0}, 
$\text{Int} \, \rho \subset \mcl{B}_{R_{0}}$ implies 
$\Hm^{a} \bigl[ \Ss' \cap (\text{Int} \, \rho) \bigr] = 0$. Therefore, by part \ref{I:no.change.off.Q} of theorem \ref{T:polyApproxThm}, we have 
$\Hm^{a} \bigl[ \tilde{\Ss} \cap (\text{Int} \, \rho) \bigr] = 0$ as well. 
Hence, by \eqref{E:S.tilde-|Q|.exprssn}, 
    \begin{equation}  \label{E:Hma(S.tilde-|Q|)=0}
      \Hm^{a} \bigl( \tilde{\Ss} \setminus |Q| \bigr) = 0 .
    \end{equation}

It then follows from \eqref{E:Hma(S.tilde)>0} that
	\begin{equation} \label{E:a.leq.dim.S.tilde.cap.Q}
	    0 < \Hm^{a} (\tilde{\Ss}) = \Hm^{a} \bigl( \tilde{\Ss} \setminus |Q| \bigr) 
	      + \Hm^{a} \bigl( \tilde{\Ss} \cap |Q| \bigr) 
	         = \Hm^{a} \bigl( \tilde{\Ss} \cap |Q| \bigr).
	  \end{equation}
Thus, $\dim (\tilde{\Ss} \cap |Q|) \geq a$. But by theorem \ref{T:polyApproxThm}, part \ref{I:S.tilde.subcomp} and \eqref{E:Haus.dim.of.cmplx}, 
we have that $\dim (\tilde{\Ss} \cap |Q|)$ is an integer $\leq a$. Therefore, by \eqref{E:Hma(S.tilde)>0}, we have  
	\begin{equation}  \label{E:a.is.integer}
		\dim (\tilde{\Ss} \cap |Q|) = a \text{ and } a \text{ is an integer,}
	\end{equation}  
as asserted in the theorem.\footnote{Actually, we have only proved \eqref{E:a.is.integer} assuming \eqref{E:Hma.S'.cap.B.eta.R0.0} holds. Normally, we want $R_{0}$ be be independent of $(\Phi, \Ss')$, but for the purposes of proving \eqref{E:a.is.integer}, take $R_{0} > 0$ so small that for some $R$ 
as in \eqref{E:essntl.dist from.S.to.Pf.small}, we have $R_{0} \leq R$ so, by \eqref{E:R.R0.BR0}, \eqref{E:Hma.S'.cap.B.eta.R0.0} \emph{does} hold.} 
(To prove \eqref{E:a.is.integer} we relied on theorem \ref{T:polyApproxThm}, which required 
$\dim \bigl( \Ss' \cap |Q| \bigr) \leq a$, which is a consequence of \eqref{E:dim.S'.=.a}, which is a consequence of \eqref{E:Hm.a.S.finite}.)

Now remember that $\D$ and $|P|$ are distinct. Let 
    \begin{equation}  \label{E:triangulation.Lip.const}
      K_{f^{-1}} \text{ be the Lipschitz constant of the inverse triangulation } f^{-1} : \D \to |P| .
    \end{equation}
By \eqref{E:polyhedral.volume.magnification.factor}, \eqref{E:Hma(S.tilde-|Q|)=0}, 
and \eqref{E:Lip.magnification.of.Hm}, there exists $K(R_{0}) := K \in (0, \infty)$ s.t.\ 
    \begin{align}  \label{E:Ha.S.tilde.leq.mult.Ha.S'}
      \Hm^{a}(\tilde{\Ss}) &= \Hm^{a}(\tilde{\Ss} \setminus |Q|) 
          + \Hm^{a}(\tilde{\Ss} \cap |Q|) \notag \\ 
        &= \Hm^{a}(\tilde{\Ss} \cap |Q|) \notag \\
        &\leq K(R_{0}) \Hm^{a} \bigl[ f^{-1}(\Ss') \cap |Q| \bigr] \\
        &\leq K(R_{0}) \Hm^{a} \bigl[ f^{-1}(\Ss') \bigr] \notag \\
        &\leq K(R_{0}) K_{f^{-1}}^{a} \Hm^{a} (\Ss') \notag . 
    \end{align} 

Since $\dim (\tilde{\Ss} \cap |Q|) = a$, we have that $\tilde{\Ss} \cap |Q|$ contains at least one simplex of dimension $a$. Let
     \begin{equation}  \label{E:Vol.a(R0).defn}
      Vol_{a}(R_{0}) > 0 \text{ be the  volume of the smallest $a$-simplex in } P .
    \end{equation}
(By \eqref{E:complex.P.is.finite}, $P$ is finite. Recall that, by \eqref{E:diam.simps.<.R0/3}, this volume depends on $R_{0}$. In remark \ref{R:topological.bounds.on.vol}, another choice of ${E:Vol.a(R0).defn}$ is suggested.) Then, by \eqref{E:Ha.S.tilde.leq.mult.Ha.S'}, 
    \begin{equation}  \label{E:Vol.a(R0).bnd.on.Hm.a(S')}
       \Hm^{a} (\Ss') \geq K(R_{0})^{-1} K_{f^{-1}}^{-a} \Hm^{a}(\tilde{\Ss}) 
         \geq K(R_{0})^{-1} K_{f^{-1}}^{-a} Vol_{a}(R_{0} .
    \end{equation}
    
Define,  
	\begin{equation}  \label{E:gamma.R0.defn}
		\gamma(R_{0}) 
		  := K_{f^{-1}}^{-a} K(R_{0})^{-1} \, Vol_{a}(R_{0}) 
		    \cdot (diam(\D))^{-\min(d-p-1, a)} . 
	\end{equation}
Then, if $\gamma \leq \gamma(R_{0})$ and $R_{0}$ satisfies \eqref{E:outside.B.toT.>.R0} and depends only on $\bigl(\D, C[\Pf], \T, a, \F \bigr)$, e.g. as in \eqref{E:R0.choice}, then, by \eqref{E:Vol.a(R0).bnd.on.Hm.a(S')}, \eqref{E:Hm.a.S.geq.R.d-p-1} holds for 
$R \geq R_{0}$.

\subsection{Small $R$: Rectifiability}   \label{SS:Small.R.rect}
The argument given in the last section 
yields a value, $\gamma(R_{0})$, of $\gamma$ in \eqref{E:Hm.a.S.geq.R.d-p-1} that depends on $R_{0}$ satisfying \eqref{E:Hma.S'.cap.B.eta.R0.0}. But we want a value of $\gamma$ that does not depend on $R \geq R_{0}$  and, hence, on $R$ in \eqref{E:essntl.dist from.S.to.Pf.small}. The argument in the last section depends on finding a subcomplex, $Q$, of $P$ satisfying \eqref{E:Q.stays.away.from.T}. As $R_{0} \downarrow 0$, constructing such a $Q$ requires ever finer subdivisions of $P$. As the subdivisions get finer the number 
$Vol_{a}(R_{0}) > 0$ gets smaller, hence, $\gamma(R_{0})$ gets smaller. So a different argument is needed to get a $\gamma$ independent of $R$. We will use "dilation" to reduce the general case to that discussed in the last section. In order to carry out that operation, we need to replace $\Ss'$ by a set whose behavior under dilation is easier to study.

\emph{Suppose the theorem is false.} Then there exist sequences, 
$\{ \Phi_{m}, \; m = 1, 2, \ldots \}$ and 
$\{ R_{m} \}$ of $G$-invariant data maps and positive numbers, resp., having the following properties.  
    \begin{property}  \label{Pty:four.properties}
          \begin{enumerate}
          \item For each $m$ there exists a $G$-invariant closed set 
          $\Ss_{m}' \subset \D$ and continuous map 
             $\Phi_{m} : \D \setminus \Ss_{m}' \to \F$.
          			\label{I:sing.set.in.Sm'}
          \item Property \ref{Pty:agree.near.T} holds with $\Phi_{m}$ in place of $\Phi$ 
          and $\Ss_{m}'$ 
             in place of $\Ss'$. \label{I:pty:agree.near.T.holds.forPhi.m}
          \item For each $m$ we have $dist^{a}(\Ss_{m}', \Pf) \geq R_{m} > 0$. (
          See \eqref{E:essential.dist.defn} for definition of $dist^{a}$.)  
                              \label{I:Rm.away.from.Pf}
          \item $\infty > R_{m}^{ -\min( d-p-1, a) } \, \Hm^{a}(\Ss_{m}') \to 0$. 
             \label{I:Hm.Sm.to.0.fast}
       \end{enumerate}
    \end{property}
By property \ref{I:pty:agree.near.T.holds.forPhi.m}, we have $\Hm^{a}(\Ss_{m}') > 0$ for every $m$. Now, by \eqref{E:D.is.cmpct.Riem.manif}, $\D$ is compact with positive dimension. Therefore, the same is true of $|P|$. 
But $0 < R_{m} \leq dist^{a}(\Ss_{m}', \Pf) \leq diam \bigl( |P| \bigr)$. Hence, by properties \ref{I:Rm.away.from.Pf} and \ref{I:Hm.Sm.to.0.fast}, we have that 
    \begin{equation}  \label{E:Rm.bounded}
        \{ R_{m} \} \text{ is bounded in $m$ so } \Hm^{a}(\Ss_{m}') \to 0 \text{ as } m \to \infty.
    \end{equation}
    
Suppose there exists $R_{0} > 0$, possibly depending on the whole sequence 
$\{ \Phi_{m}, \Ss_{m}' \}$, s.t.\ \eqref{E:outside.B.toT.>.R0} holds and, for an infinite collection, $M$, of $m$'s, we have $R_{m} > R_{0}$.  
Let $x \in \Ss_{m}'$ and suppose $x \in \mcl{B}_{R_{0}}$. 
By \eqref{E:Bs.w/in.s.of.P} and part \ref{I:Rm.away.from.Pf} of property \ref{Pty:four.properties},  
    \begin{equation*}
         dist(x, \Pf) \leq R_{0} < R_{m} \leq dist^{a}(\Ss_{m}', \Pf) , \quad m \in M .
    \end{equation*}
Therefore, by \eqref{E:essential.dist.defn}, if $m \in M$ we have that \eqref{E:Hma.S'.cap.B.eta.R0.0} holds 
with $\Ss'$ replaced $\Ss_{m}'$. But, by \eqref{E:Q.defn}, 
$|P| \setminus \mcl{B}_{R_{0}} \subset |Q|$. Thus, 
$\Hm^{a}(\tilde{\Ss}) = \Hm^{a} \bigl( \tilde{\Ss} \cap |Q| \bigr)$.  
Hence, by \eqref{E:polyhedral.volume.magnification.factor}, 
$\bigl\{ \Hm^{a}(\Ss_{m}') : m \in M \bigr\}$ is bounded away from 0. This contradicts \eqref{E:Rm.bounded}. Therefore, lopping off the beginning of the sequence if necessary, WLOG we may assume 
	\begin{multline}  \label{E:Rm.=.dist.Sm.P}
		 1 > dist^{a}(\Ss_{m}', \Pf) \to 0, \text{ and, hence, } R_{m} \to 0, 
		   \text{ as } m \to \infty. \\
		     \text{ In particular, } R_{m} \in (0,1) \text{ for all } m = 1, 2, \ldots.
	\end{multline}

Suppose for some $m$ we have $\Hm^{a} (\Ss_{m}') = 0$. Then, by \eqref{E:essential.dist.defn}, $dist^{a}(\Ss_{m}', \Pf) = D$. But, by assumption, $0 \leq t < d$. Therefore, $diam \, \D > 0$. So, by \eqref{E:Rm.=.dist.Sm.P}, eventually 
$\Hm^{a} (\Ss_{m}') > 0$. Hence, by \eqref{E:Rm.bounded}, we may also assume
    \begin{equation} \label{E:Ha.Sm'.finite}
        0 < \Hm^{a} (\Ss_{m}') < \infty \text{ for all } m.
    \end{equation}
Therefore, by \eqref{E:Haus.dim.defn}, we may assume that 
	\begin{equation} \label{E:for.all.m.dim.Sm.=.a}
		\dim \Ss_{m}' = a \text{ for all } m.
	\end{equation}
 
Recall the definition of countable rectifiability (Hardt and Simon \cite[Definition 2.1', p.\ 20]{rHlS86.GMT}). \emph{Claim:} 
     \begin{multline} \label{E:Sm'.countably.rect}
	   \text{We may assume that } a \text{ is an integer and } \\
	     \text{ for every } m, \; \Ss_{m}' \text{ is countably $a$-rectifiable and has finite }
	        \Hm^{a}\text{-measure. } \\
	 \text{ Moreover, $\bigl\{ (\Phi_{m}, \Ss_{m}', R_{m}) \big\}$ has properties 
	     \ref{I:sing.set.in.Sm'} through \ref{I:Hm.Sm.to.0.fast}.}
     \end{multline}

Let $m = 1,2, \ldots$ be arbitrary, but fixed. Reasoning as above, by \eqref{E:Bs.w/in.s.of.P} and part \ref{I:Rm.away.from.Pf} of property \ref{Pty:four.properties},
we have that \eqref{E:Hma.S'.cap.B.eta.R0.0} holds with $\Ss'$ replaced $\Ss_{m}'$ and $R_{m}$ in place of $R_{0}$:
	\begin{equation}   \label{E:Hma.S'.cap.B.Rm.0}
		\Hm^{a} (\Ss_{m}'  \cap \mcl{B}_{R_{m}}) = 0.
	  \end{equation}

By \eqref{E:Bs.contains.nbhd.of.T} and compactness of $\T$, there exists 
$r = r_{m} \in (0, R_{m}]$ s.t.\  
	\begin{equation} \label{E:outside.B.Rm.toT.>.r}
		dist(\D \setminus \mcl{B}_{R_{m}/2}, \T) > r_{m}. 
	\end{equation}
	
By assumption and \eqref{E:complex.P.is.finite}, $\D$ has a bi-Lipschitz triangulation 
$f : |P| \to \D$, where $P$ is a finite simplicial complex. WLOG and by assumption we may temporarily identify $\D$ with $|P|$ and assume $G$ is a group of simplicial homeomorphisms on $P$. We argue in a similar way to the proof of \eqref{E:Q.stays.away.from.T} with $R_{m}/2$ in place of $R_{0}$. If necessary, replace $P$ by a $G$-invariant subdivision $P_{m} := P'$ of the sort described in theorem \ref{T:polyApproxThm}, part \ref{I:arb.fine.subdivision} (recall \eqref{E:use.G.invar.version.of.thm}) so that we may assume 
	\begin{equation}  \label{E:all.simps.have.small.diam}
		\text{All simplices in } P_{m} \text{ have diameter } < r/8 \leq R_{m}/8.
	\end{equation}
Let $Q_{m}$ be the subcomplex of $P_{m}$ consisting of all simplices that intersect 
$\D \setminus \mcl{B}_{R_{m}/2}$ and all faces of all such simplices. Usually we identify $P$ and $\D$. We have  
    \begin{equation*}
      \D \setminus \mcl{B}_{R_{m}/2} \subset |Q_{m}| 
        \text{ so } P \setminus |Q_{m}| \subset \mcl{B}_{R_{m}/2} . 
    \end{equation*}
We prove the following analogue of \eqref{E:Q.stays.away.from.T}.
	\begin{multline}  \label{E:for.each.m.Q.aint.close.to.T}
		|Q_{m}| \text{ does not intersect any simplex} \\ 
		  \text{that in turn intersects a simplex intersecting } \T.
	\end{multline}
The proof is similar to that of \eqref{E:Q.stays.away.from.T}. 
Let $\sigma \in P_{m}$ have nonempty intersection with $|Q_{m}|$. Then there exists 
$\tau \in Q_{m}$ s.t.\ $\tau \setminus \mcl{B}_{R_{m}/2} \neq \varnothing$ 
and $\sigma \cap \tau \ne \varnothing$. Suppose $\zeta \in P$ has nonempty intersection with $\T$ and with $\sigma$. 
Let $x_{1} \in \tau \setminus \mcl{B}_{R_{m}/2}$, 
$x_{2} \in  \sigma \cap \tau$, $x_{3} \in \sigma \cap \zeta$, and $x_{4} \in \zeta \cap \T$.
Then, using \eqref{E:outside.B.Rm.toT.>.r} and \eqref{E:all.simps.have.small.diam}, we arrive at
    \begin{equation*}
    	r < \xi(x_{1}, x_{4}) \leq \xi(x_{1}, x_{2}) + \xi(x_{2}, x_{3}) + \xi(x_{3}, x_{4}) 
	  < \frac{3}{8} r.
    \end{equation*}
Contradiction. This proves \eqref{E:for.each.m.Q.aint.close.to.T}.

Let $(\tilde{\Phi}_{m}, \tilde{\Ss}_{m})$ be to $(\Phi_{m}, \Ss_{m}')$ as 
$(\tilde{\Phi}, \tilde{\Ss})$ is to $(\Phi, \Ss')$ in proposition \ref{P:Phi.invar.Phi.tilde.invar}. In particular, theorem \ref{T:polyApproxThm} and \eqref{E:for.all.m.dim.Sm.=.a} tell us  
    \begin{equation}  \label{E:dim.S.tilde.leq.a.tildes.equivar}
      \tilde{\Ss}_{m} \text{ is closed, } \dim \tilde{\Ss}_{m} \leq \dim \Ss_{m}' = a , 
        \text{ and } (\tilde{\Phi}_{m}, \tilde{\Ss}_{m}) \text{ is $G$-invariant.} 
    \end{equation}
We prove the following. 
	\begin{subequations} \label{E:S.tilde.properties}
		\begin{gather}  
			\Hm^{a}(\tilde{\Ss}_{m}) < \infty. \label{E:S.tilde.properties1} \\
			a \text{ is an integer and } \tilde{\Ss}_{m} \text{ is countably } 
			  a\text{-rectifiable.}  \label{E:S.tilde.properties2} \\
			\begin{split}
			\text{For some sequence } \tilde{R}_{m} \to 0 \text{ we have that }
			        \bigl\{ (\tilde{\Phi}_{m}, \tilde{\Ss}_{m}, \tilde{R}_{m}) \bigr\} \\ 
			        \text{ enjoys property \ref{Pty:four.properties}.} 
			\end{split} \label{E:S.tilde.properties3} 
		\end{gather}
	\end{subequations}
Proving these three statements will prove the claim \eqref{E:Sm'.countably.rect}. 
(\eqref{E:S.tilde.properties1}, $\Hm^{a}(\tilde{\Ss}_{m}) < \infty$ can be proved just as \eqref{E:Ha.Sm'.finite} is, but we prove it independently below.)

It is immediate from theorem \ref{T:polyApproxThm} that $(\tilde{\Phi}_{m}, \tilde{\Ss}_{m})$ satisfies part \ref{I:sing.set.in.Sm'} of property \ref{Pty:four.properties}. \emph{Claim:} 
	\begin{equation}   \label{E:R.tilde.m.geq.Rm/2}
		+ \infty > diam(\D) \geq \tilde{R}_{m} := dist^{a}(\tilde{\Ss}_{m}, \Pf) 
		  \geq R_{m}/2 > 0.
	\end{equation}
(In particular, $(\tilde{\Phi}_{m}, \tilde{\Ss}_{m},\tilde{R}_{m})$ satisfies part \ref{I:Rm.away.from.Pf} of property \ref{Pty:four.properties}.) 
Obviously, since part \ref{I:Rm.away.from.Pf} of property \ref{Pty:four.properties} tells us that $R_{m}/2 > 0$, only $dist^{a}(\tilde{\Ss}_{m}, \Pf) \geq R_{m}/2$ requires proof. To prove \eqref{E:R.tilde.m.geq.Rm/2}, note that by \eqref{E:essential.dist.defn} and \eqref{E:x.in.exctly.1.simplex.intrr}, it suffices to show that 
	\begin{equation} \label{E:Ha.tau.near.P.0}
		\Hm^{a} \bigl[ \tilde{\Ss}_{m} \cap (\text{Int} \, \tau) \bigr] = 0
		        \text{ for every }  \tau \in P_{m} \text{ s.t.\ } dist( \tau, \Pf ) \leq R_{m}/2.
	\end{equation}

Let $\tau$ be an arbitrary simplex in $P_{m}$ s.t.\ $dist( \tau, \Pf ) \leq R_{m}/2$. So there exists $y \in \tau$ and $z \in \Pf$ s.t.\ $\xi(y, z) < \tfrac{5}{8} R_{m}$. Let $\sigma$ be an arbitrary simplex in $P_{m}$ having $\tau$ as a face. $\sigma = \tau$ is a possibility. Let $x \in \sigma$. Then, by \eqref{E:all.simps.have.small.diam}, we have 
	\begin{equation}  \label{E:all.tau.closer.than.3/4.Rm}
		dist(x, \Pf) \leq \xi(x, y) + \xi(y, z)  < \tfrac{1}{8} R_{m} + \tfrac{5}{8} R_{m} 
		  < R_{m}.
	\end{equation}
Therefore, by part \ref{I:Rm.away.from.Pf} of property \ref{Pty:four.properties} 
 we have  $\Hm^{a} \bigl( \Ss_{m}' \cap \sigma \bigr) = 0$. In particular, \linebreak
 $\Hm^{a} \bigl( \Ss_{m}' \cap \tau \bigr) = 0$. 
 
Suppose $\tau \notin Q_{m}$. Then, since $\Hm^{a} \bigl( \Ss_{m}' \cap \tau \bigr) = 0$, by
part \ref{I:no.change.off.Q} of theorem \ref{T:polyApproxThm}, \eqref{E:Ha.tau.near.P.0} holds for $\tau$.
Next, suppose $\tau \in Q_{m}$, $dist( \tau, \Pf ) \leq R_{m}/2$, but $\Hm^{a} \bigl[ \tilde{\Ss}_{m} \cap (\text{Int} \, \tau) \bigr] > 0$. Then, by part \ref{I:Ha.tau.=0.if.Ha.all.nbrs.0} of theorem \ref{T:polyApproxThm}, there exists $\sigma \in Q_{m}$ s.t.\ $\tau$ is a face of $\sigma$ and 
$\Hm^{a} \bigl[ \Ss_{m}' \cap (\text{Int} \, \sigma) \bigr] > 0$. But we just proved that $\Hm^{a} \bigl( \Ss_{m}' \cap \sigma \bigr) = 0$. This contradiction completes that proof of \eqref{E:Ha.tau.near.P.0} and the claim \eqref{E:R.tilde.m.geq.Rm/2}. 

For $a \geq 0$, let 
    \begin{equation}  \label{E:integer.part.floor}
      \lfloor a \rfloor \text{ denote the integer part of $a$, the largest integer } \leq a .
    \end{equation}
By part \ref{I:S.tilde.subcomp} of theorem \ref{T:polyApproxThm}, we have,
	\begin{multline} \label{E:for.all.m.dim.Sm.tilde.is.complex}
		\text{For all } m, \, \tilde{\Ss}_{m} \cap |Q_{m}| 
		  \text{ is either empty or the underlying space} \\
		    \text{ of a subcomplex of the } \lfloor a \rfloor\text{-skeleton of } Q_{m}. 
	\end{multline}
Let $x \in \tilde{\Ss}_{m} \cap \mcl{B}_{R_{m}/2}$. Then, by \eqref{E:dist.outside.of.Bs.to.P}, we have $dist(x, \Pf) < R_{m}/2$. 
Since, by definition of $Q$, $P \setminus |Q_{m}| \subset \mcl{B}_{R_{m}/2}$, we have, by 
\eqref{E:R.tilde.m.geq.Rm/2},  
	\begin{equation}  \label{E:Ha.Sm.tilde.less.Qm.=.0}
		\Hm^{a}(\tilde{\Ss}_{m} \setminus |Q_{m}|) = 0.
	\end{equation} 
Therefore, by \eqref{E:polyhedral.volume.magnification.factor} and \eqref{E:Ha.Sm'.finite}, 
$\Hm^{a}(\tilde{\Ss}_{m}) < \infty$, \emph{proving \eqref{E:S.tilde.properties1}}.

Arguing as in the proof of \eqref{E:S.tilde.has.empty.intrr}, it follows from this that 
$\tilde{\Ss}_{m}$ has empty interior. 

We prove that part \ref{Pty:agree.near.T} of property \ref{Pty:four.properties} applies 
to $(\tilde{\Phi}_{m}, \tilde{\Ss}_{m})$. I.e., we prove part \ref{I:pty:agree.near.T.holds.forPhi.m} 
of $\{ \Phi_{m}, \Ss_{m}' \}$ but for $(\tilde{\Phi}_{m}, \tilde{\Ss}_{m})$. 
Let $\tau \in P_{m}$ and suppose $\tau \cap \T \neq \varnothing$. Then by  \eqref{E:for.each.m.Q.aint.close.to.T} and, trivially, part \ref{I:no.change.off.nbhd.of.S.in.Q} of theorem \ref{T:polyApproxThm}, we have 
$\tilde{\Ss}_{m} \cap \tau = \Ss_{m}' \cap \tau$ and $\tilde{\Phi}_{m}$ and $\Phi_{m}$ agree on $\tau \setminus \Ss_{m}'$. It follows that $\tilde{\Ss}_{m} \cap \T = \Ss_{m}' \cap \T$ 
and $\tilde{\Phi}_{m}$ and $\Phi_{m}$ agree on $\T \setminus \Ss_{m}'$. From proposition \ref{P:Phi.invar.Phi.tilde.invar} we have that $\tilde{\Phi}_{m}$ 
and $\tilde{\Ss}_{m}$ are $G$-invariant. By \eqref{E:dim.S.tilde.leq.a.tildes.equivar}, 
$\tilde{\Ss}_{m}$ is closed. Therefore, by remark \ref{R:consequences.of.property}, 
$(\tilde{\Phi}_{m}, \tilde{\Ss}_{m}, G, \T, a)$ inherits property \ref{Pty:agree.near.T} 
from $(\Phi_{m}, \Ss_{m}', G, \T, a)$. Thus, 
$\bigl\{ (\tilde{\Phi}_{m}, \tilde{\Ss}_{m}, \tilde{R}_{m}) \bigr\}$ satisfies part \ref{I:pty:agree.near.T.holds.forPhi.m} of property \ref{Pty:four.properties}. 
In particular, $\Hm^{a} (\tilde{\Ss}_{m}) > 0$. In addition, since $|P_{m}|$ is bounded we see that $\{ \tilde{R}_{m} \}$ is bounded. 

Since $(\tilde{\Phi}_{m}, \tilde{\Ss}_{m}, G, \T, a)$ has property \ref{Pty:agree.near.T}, we have  
$\dim (\tilde{\Ss}_{m}) \geq a$. But by theorem \ref{T:polyApproxThm} part \ref{I:dim.Stilde.no.bggr.thn.dim.S} and \eqref{E:for.all.m.dim.Sm.=.a} we have 
$\dim (\tilde{\Ss}_{m}) \leq \dim (\Ss'_{m}) = a$. I.e., $\dim (\tilde{\Ss}_{m}) = a$. But \eqref{E:Ha.Sm.tilde.less.Qm.=.0} and part \ref{I:S.tilde.subcomp} of theorem \ref{T:polyApproxThm} (and \eqref{E:dim.of.whole.=.max.dim.of.parts}), $\dim (\tilde{\Ss}_{m})$ is an integer. I.e., $a$ is an integer.

\eqref{E:Ha.Sm.tilde.less.Qm.=.0} together with \eqref{E:for.all.m.dim.Sm.tilde.is.complex} further imply that, except for a  $\Hm^{a}$-null set, 
$\tilde{\Ss}_{m}$ is the polytope of a simplicial complex of dimension no greater than $a$ and, hence, is $a$-rectifiable, finitely, hence, countably. We just saw that $a$ is an integer. \emph{This proves statement \eqref{E:S.tilde.properties2}.} 

By \eqref{E:R.tilde.m.geq.Rm/2} and \eqref{E:Ha.Sm.tilde.less.Qm.=.0}, we have, 
by part \ref{I:arb.fine.subdivision} of theorem \ref{T:polyApproxThm} 
that there is a constant $K < \infty$ depending only on $a$ and $P$ s.t.\
	\begin{align}  \label{E:Rm.power.Ha(Sm.tilde).to.0}
		\tilde{R}_{m}^{ -\min( d-p-1, a) } \Hm^{a}(\tilde{\Ss}_{m}) 
		     &\leq 2^{ \min( d-p-1, a) } 
		        R_{m}^{ -\min( d-p-1, a) } \Hm^{a}(\tilde{\Ss}_{m}) \notag \\
		    &= 2^{ \min( d-p-1, a) } R_{m}^{ -\min( d-p-1, a) } 
		      \Hm^{a} \bigl( \tilde{\Ss}_{m} \cap |Q_{m}| ) \\
		    &\leq 2^{ \min( d-p-1, a) } K R_{m}^{ -\min( d-p-1, a) } \
		      \Hm^{a}(\Ss_{m}'  \cap |Q_{m}| ) \notag \\
		    &\leq 2^{ \min( d-p-1, a) } K R_{m}^{ -\min( d-p-1, a) } 
		      \Hm^{a}(\Ss_{m}') \to 0 \text{ as } m \to \infty. \notag 
	\end{align}
Therefore, $\bigl\{ (\tilde{\Phi}_{m}, \tilde{\Ss}_{m}, \tilde{R}_{m}) \bigr\}$ satisfies part \ref{I:Hm.Sm.to.0.fast} of property \ref{Pty:four.properties}. 
We have already established that it satisfies parts \ref{I:sing.set.in.Sm'}, \ref{I:pty:agree.near.T.holds.forPhi.m}, and \ref{I:Rm.away.from.Pf} of property \ref{Pty:four.properties}. Thus, \emph{\eqref{E:S.tilde.properties3} is proved.} We already proved \eqref{E:S.tilde.properties1} and \eqref{E:S.tilde.properties2}. \emph{This means we have proved the claim \eqref{E:Sm'.countably.rect}.}

\subsection{Small $R$: ``Dilation''}  \label{SS:Dilation}
Let $\{ \Phi_{m}, \Ss_{m}', R_{m} \}$ be as in section \ref{SS:Small.R.rect} and assume \eqref{E:Sm'.countably.rect} holds. Next, we ``dilate'' $\Ss_{m}'$ so that 
$\Hm^{a}$-almost none of it lies in $\mcl{B}_{\ess} := Exp (C_{\ess})$. (See \eqref{E:Bs.defn} and \eqref{E:Cs.defn}.) Let $R \in (0,\ess)$. We define a map 
$f_{dilate, R} : \D \to \D$. Recall that, by \eqref{E:eps.P.geq.2.on.U}, 
we have $\epsilon_{\Pf} \restriction_{\clU} \geq 2$. Therefore, 
if $(x,v) \in C_{1}$ then, by \eqref{E:rho.in.[0,1]} and \eqref{E:Cs.defn},   
for every $r \in [0,2]$ we have $(x, r v) \in C_{2} \subset C[\overline{\clU}]$. (See \eqref{E:C.sub.s.in.Bold.F.[0,s)} and \eqref{E:overline.U.defn}.) Let 
	\begin{equation} \label{E:AB.dilate.defn}
		A_{dilate,R} := \frac{\threeess}{\sixess-R} \in (1/2, 1)  
		  \text{ and }
		    B_{dilate,R} := 2(1 - A_{dilate,R}) 
		      = \twoess \frac{\threeess-R}{\sixess - R} \in (0,1),
	\end{equation}  
where $\threeess \in [1, \infty)$ is a constant, not depending on $R$ or $\{ \Phi_{m}, \Ss_{m}', R_{m} \}$, that is described in \eqref{E:threeess.appears}. In particular,
    \begin{equation}   \label{E:R<threeess}
      R < 1 \leq \threeess.
    \end{equation}

Write $A := A(R) := A_{dilate,R}$ and $B := B(R) := B_{dilate,R}$. Then
	\begin{align}  \label{E:R.<.threeess/3.A.B.ineqs}
		\text{If } 0 < R < \threeess/3& \text{ then: } \notag \\
		\begin{split}
			&A \in (1/2, 3/5), \\
			&B \in (4/5, 1), \\
			&A/B \in (1/2, 3/4), \\
			&A + B \in (7/5, 3/2), \text{ and } \\
			&2A+B = 2.
		\end{split}
	\end{align} 

Define $F = F_{R} : C[\Pf] \to C[\Pf]$ as follows:
	\begin{multline} \label{E:F.on.C[Pf].defn}
		F_{R}(y,v) \\ :=
			\begin{cases}
			      \bigl( y,(\threeess/R) v \bigr), &\text{ if } |v| < R \rho(y) /\threeess, \\
				\Bigl( y, \bigl[ A_{dilate,R} |v| 
				          + B_{dilate,R} \, \rho(y) \bigr] |v|^{-1} v \Bigr), 
				     &\text{ if } R \rho(y)/\threeess \leq |v| < 2 \rho(y), \\
			     (y,v), &\text{ if } |v| \geq 2 \rho(y) , 
			\end{cases} \\
		  \quad (y, v) \in C[\Pf]. 
	\end{multline}
Figure \ref{F:F.plot} shows the graph of the length of an example of an $F$. 
       \begin{figure}
             \includegraphics[width=6in, height = 5in, , angle = 0]{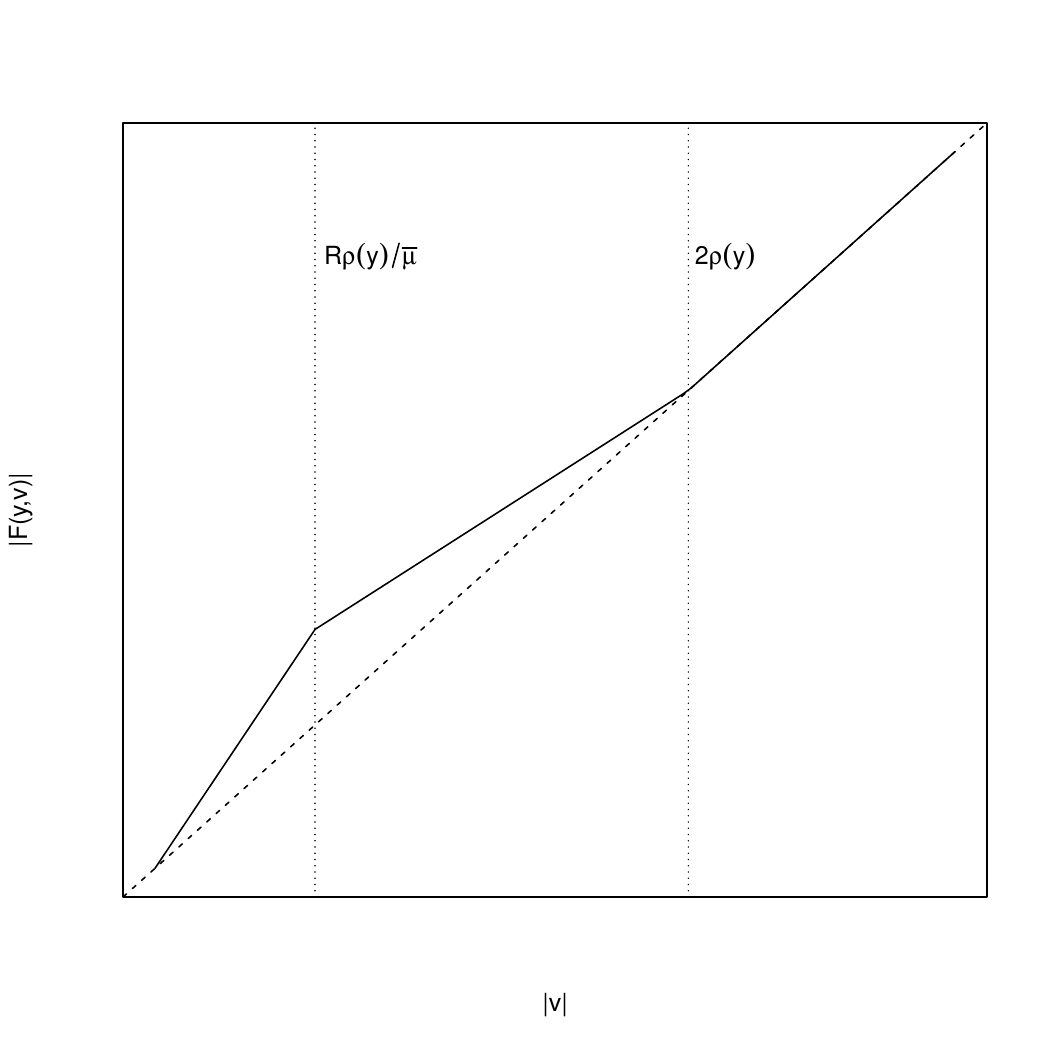}
          \caption{Graph of $|F|$ versus $|v|$ for a hypothetical $R$, $\bar{\mu}$, 
          and $\rho(y)$. Dashed line is the identity $ordinate = abscissa$ (``$y=x$'').}
\label{F:F.plot}
       \end{figure} 

Note that, by \eqref{E:Cs.defn}, $|v| < R \rho(y) /\threeess$ is equivalent to 
$(y,v) \in C_{R/\threeess}$ and $|v| < 2 \rho(y)$ is equivalent to 
$(y,v) \in C_{2}$. The condition $R \rho(y)/\threeess \leq |v| < 2 \rho(y)$ is never satisfied if $\rho(y) = 0$ so $|v|^{-1}$ makes sense if that condition is satisfied. On the other hand, by \eqref{E:rho.on.W.off.U}, 
    \begin{equation}  \label{E:"f.dil(y,v)=(y,v)".if.y.notin.U}
      \text{If $y \in \Pf \setminus \clU$ then $|v| \geq 2 \rho(y) = 0$ so 
        $F(y,v) = (y,v)$ if } 
          y \in \Pf \setminus \clU .
    \end{equation}

Now we define a map $f_{dilate, R} : \D \to \D$. (Do not confuse $f_{dilate, R}$ and its nicknames with the triangulation $f : |P| \to \D$.)
    \begin{equation}  \label{E:when.f.dilate.is.identity}
      \text{For } x \in \D \setminus \mcl{C} \text{ define } f_{dilate, R}(x) := x .
    \end{equation}
If $x \in \mcl{C}$ define $f_{dilate, R}(x) : \D \to \D$ as follows. Recall $\alpha : \mcl{C} \to C[\Pf]$ (definition \ref{D:fibering.by.cones} part \ref{I:Exp.alpha.homeom}). Recall also definition \ref{D:fibering.by.cones}. 
    \begin{equation}  \label{E:f.in.terms.of.F}
      f_{dilate, R}(x) = Exp \circ F \circ \alpha ,  \quad x \in \mcl{C} .
    \end{equation}
By \eqref{E:Cs.defn}, 
$C[\Pf] \setminus C_{2} = \bigl\{ (y,v) \in C[\Pf] : |v| \geq 2 \rho(y) \bigr\}$.

We have
    \begin{equation}  \label{E:F.from.f.dil}
      F := \alpha \circ f_{dilate, R} \circ \alpha^{-1} \text{ on } C[\Pf]. 
    \end{equation}

Let $y \in \Pf$. Supppose $\rho(y) > 0$. By \eqref{E:rho.strictly.poz.on.U} and  \eqref{E:rho.on.W.off.U}, this is equivalent to $y \in \clU$. By \eqref{E:C.sub.s.in.Bold.F.[0,s)}, this happens if $y \in \T$. In any case, $\rho(y) > 0$ means $0 < R \rho(y)/\threeess$. Therefore, by \eqref{E:F.on.C[Pf].defn}, we have $F(y,0) = (y,0)$. 
    \begin{equation}  \label{E:f.dil.ident.on.U}
      f_{dilate, R} \text{ is the identity on } \clU.
    \end{equation} 
(See \eqref{E:fdil.is.identity.on.Pf} and \eqref{E:f.dil.ident.on.U} below.) 

Let $(y,v) \in C[\Pf]$. Suppose $|v| = 2 \rho(y)$. Then, by \eqref{E:Cs.defn}, $(y,v)$ is a boundary point of $C_{2}$. By \eqref{E:R.<.threeess/3.A.B.ineqs},  the limit of $F(y', w)$ as $y' \to y$ through $\Pf$ and $w \to v$ with $|w| < 2 \rho(y)$ is 
$(y,v) = F(y,v)$. Because, by \eqref{E:R.<.threeess/3.A.B.ineqs}, the vector part of $F(y',w)$ converges to
    \begin{multline*}
      \bigl[A_{dilate,R} |v| + B_{dilate,R} \, \rho(y) \bigr] |v|^{-1} v
        =  \bigl[ 2 A_{dilate,R} \, \rho(y) + B_{dilate,R} \, \rho(y) \bigr] |v|^{-1} v \\
          = \bigl[ 2 A_{dilate,R} + B_{dilate,R} \bigr] \rho(y) |v|^{-1} v
            = 2 \rho(y) |v|^{-1} v = 2 \rho(y) \bigl( 2 \rho(y) \bigr)^{-1} v = v . 
    \end{multline*}
So $F$ is continuous at $(y, v)$ with $|v| = 2 \rho(y)$. 
 
Now suppose $|v| = R \rho(y)/\threeess$. Then the limit of $F(y', w)$ as $y' \to y$ 
through $\Pf$ and $w \to v$ with $|w| < 2 \rho(y)$ is $(R \rho(y) /\threeess) v$. Recall \eqref{E:varkappa.is.other.proj}. 
By \eqref{E:F.on.C[Pf].defn} and \eqref{E:AB.dilate.defn}, 
    \begin{align*}
      \varkappa \circ F (y,v) = \bigl[ A |v| + B \rho(y) \bigr] |v|^{-1} v
        &= \left( \frac{ A R \rho(y) + B \rho(y) \threeess}{\threeess} \right) |v|^{-1} v \\
        &= \left( \frac{ R \threeess \rho(y) 
            + 2(\threeess-R) \threeess \rho(y)}{\threeess(2 \threeess - R)} \right) 
              |v|^{-1} v  \\
        &= \left( \frac{ 2 \threeess^{2} \rho(y) - \threeess R \rho(y) }
          {\threeess(2 \threeess - R)} \right) |v|^{-1} v \\
        &= \rho(y) |v|^{-1} v = (R \rho(y) /\threeess) v. 
    \end{align*}
Thus, if $|v| = R \rho(y)/\threeess$ then, by \eqref{E:Cs.defn}, $F(y,v)$ is a boundary point of $C_{1}$. 

Let $(y,v) \in C[\Pf]$. We conclude:
    \begin{enumerate}  
        \item $F$ is continuous on $C[\Pf]$.  \label{I:F.is.cont}
        \item $\bigl| F(y,v) \bigr|$ is strictly increasing in $|v|$. (See \eqref{E:Riem.metrics.on.D.on.Rk.same}.)  \label{I:|F|.increasing.in.|v|}
        \item Consequently, by \eqref{E:R.<.threeess/3.A.B.ineqs}, 
          \label{I:C1,C2,C}
    \begin{align}  \label{E:F.on.parts}
       F(y,v) \in C_{1} &\text{ if and only if } |v| < R \rho(y) /\threeess, \notag \\
       F(y,v) \in C_{2} \setminus C_{1} &\text{ if and only if }  
           R \rho(y)/\threeess \leq |v| < 2 \rho(y), \\
       F(y,v) \in C[\Pf] \setminus C_{2} &\text{ if and only if } (y,v) \in C[\Pf] 
          \text{ and } |v| \geq 2 \rho(y). \notag 
    \end{align}
    \end{enumerate}  
It follows from \eqref{E:F.on.parts} and \eqref{E:C.sub.s.in.Bold.F.[0,s)} that  
    \begin{equation}  \label{E:F(C[P])=C[P]}
      F \bigl( C[\Pf] \bigr) = C[\Pf] .
    \end{equation}

\emph{Claim:}
    \begin{equation}  \label{E:F.is.injective}
      F \text{ is injective.}
    \end{equation} 
The only case where this is not obvious is on $C_{2} \setminus C_{1}$. So let $(y, v), (y',v') \in C_{2} \setminus C_{1}$ and suppose $F(y, v) = F(y',v')$. Obviously $y = y'$. By property 
\ref{I:|F|.increasing.in.|v|} of $F$ above, we must have $|v| = |v'|$. Then, referring to \eqref{E:F.on.C[Pf].defn}, 
    \begin{align*}
      \bigl[A_{dilate,R} |v| + B_{dilate,R} \, \rho(y') \bigr] v
        &= |v| \bigl[ A_{dilate,R} |v| + B_{dilate,R} \, \rho(y') \bigr] |v|^{-1} v  \\
        &= |v| \bigl[ A_{dilate,R} |v'| + B_{dilate,R} \, \rho(y') \bigr] |v'|^{-1} v'  \\
        &= \bigl[ A_{dilate,R} |v| + B_{dilate,R} \, \rho(y') \bigr] v' .
    \end{align*}
By \eqref{E:AB.dilate.defn}, $A_{dilate,R} |v| + B_{dilate,R} \, \rho(y') \neq 0$. So $v = v'$. This proves the claim \eqref{E:F.is.injective}. 

By \eqref{E:when.f.dilate.is.identity}, \eqref{E:F.is.injective}, part \ref{I:Exp.alpha.homeom} of definition \ref{D:fibering.by.cones}, and \eqref{E:f.in.terms.of.F},
    \begin{equation}  \label{E:fdil.is.bijective}
      f_{dilate, R} : \D \to \D \text{ is bijective} .
    \end{equation}

Let $(y,v) \in C[\clU]$ and $(y,w) = F(y,v)$. 
By \eqref{E:F.on.parts}, we have 
$(y,v) \in C_{\twoess} \setminus C_{R/\threeess}$ if and only if 
$(y,w) \in C_{\twoess} \setminus C_{1}$ and in that case we have, 
    \begin{multline}  \label{E:F.F.invrs}
      F(y,v) = \bigl( y, A v + B \rho(y) |v|^{-1} v \bigr) 
        = \bigl( y, (A + B \rho(y) |v|^{-1}) v \bigr) \\
      \text{ and }
        F^{-1}(y,w) = \bigl( y, A^{-1} w - A^{-1} B \, \rho(y) |w|^{-1} w \bigr)
           = \bigl( y, (A^{-1} - A^{-1} B \, \rho(y) |w|^{-1} ) w \bigr).
    \end{multline}

Let $y \in \Pf$. Supppose $\rho(y) = 0$. By \eqref{E:rho.strictly.poz.on.U} and  \eqref{E:rho.on.W.off.U}, this is equivalent to $y \notin \clU$. E.g., by \eqref{E:C.sub.s.in.Bold.F.[0,s)}, this happens if $y \in \T$. In any case, $\rho(y) = 0$ means $0 \geq 2 \rho(y)$. Therefore, by \eqref{E:F.on.C[Pf].defn}, we have 
$F(y,0) = (y,0)$. Combining this with \eqref{E:f.dil.ident.on.U}, we get 
$\alpha \circ f_{dilate, R} \circ \alpha^{-1}(y,v) = F(y,0) = (y,0)$. By part \ref{I:Exp.alpha.homeom} of definition \ref{D:fibering.by.cones}, we have $\alpha(y) = y$ and $\alpha^{-1}(y,0) = y$. Therefore, recalling item \ref{I:C1,C2,C} and \eqref{E:fdil.is.bijective} above,  
    \begin{align} \label{E:fdil.is.identity.on.Pf}
        &f_{dilate, R} \text{ is the identity on } \Pf. 
          \text{ In particular, it is the identity on } \T. 
           \notag \\
          &\qquad \text{Moreover, } f_{dilate, R}( \mcl{C} ) = \mcl{C} . \\
        &\text{ The same is true of }  f_{dilate, R}^{-1} \notag .
    \end{align}

Recall the definition of $\pi_{C}$ from definition \ref{D:fibering.by.cones}. 
Let $x \in \mcl{C}$ and let $(y,v) := \alpha(x)$. 
Thus, $y = \pi_{C} \circ \alpha(x)$.  Suppose $y \in \Pf \setminus \clU$. Then, by \eqref{E:rho.on.W.off.U}, $\rho(y) = 0$. Therefore, $|v| \geq 0 = 2 \rho(y)$. Hence, by \eqref{E:f.in.terms.of.F} and \eqref{E:F.on.C[Pf].defn}, 
    \begin{equation}  \label{E:over.Pf-U.fdil.is.identity}
      \text{If $x \in \mcl{C}$ with } 
        \pi_{C} \circ \alpha(x) \in \Pf \setminus \clU \text{ then } 
          f_{dilate, R}(x) = x .
    \end{equation}

Recall \eqref{E:F.from.f.dil}. We call $f_{dilate, R}$ a ``dilation'' because, we \emph{claim}: 
    \begin{multline}  \label{E:f.dil.is.dilation}
        \alpha \circ f_{dilate, R} \circ \alpha^{-1}(y,v) = (y, rv) 
          \text{ with } r = r(y,v) \geq 1. \\
            \text{ Moreover, if } |v| \leq 2 \rho(y) \text{ then } r \leq 2 \rho(y)/|v| . 
              \text{ If } |v| \geq 2 \rho(y) \text{ then } r = 1 , \\
                \quad (y, v) \in C[\Pf] . \\
           \text{ And for some } v, r(y,v) > 1 . 
    \end{multline}
Figure \ref{F:F.plot} shows this in an example. To see this, first observe that, by \eqref{E:R<threeess}, 
$R < \threeess$. So by \eqref{E:F.on.C[Pf].defn}, if $|v| < R \rho(y) /\threeess$, 
then $r(y,v) = \threeess/R >1$. If $|v| \geq 2 \rho(y)$ then $r(y,v) = 1$. Otherwise, it suffices to show $r(y,v) \in \bigl[1, 2 \rho(y) \bigr]$ 
if $R \rho(y)/\threeess \leq |v| < 2 \rho(y)$. 
Write $r[t] := \bigl( A t + B \, \rho(y) \bigr)t^{-1}$ ($t > 0$), consistent with $r(y,v)$ with $|v|$ in this range. By \eqref{E:AB.dilate.defn}, $r[t] = A + 2(1-A)  \, \rho(y) t^{-1}$. In the range 
$R \rho(y)/\threeess \leq t < 2 \rho(y)$ we have that
$r[t]$ is minimized at $t = 2 \rho(y)$. We have
    \begin{equation*}
      r \bigl[ 2 \rho(y) \bigr] = A + 2(1-A)  \, \rho(y) \frac{1}{2 \rho(y)} 
        = A + (1-A) = 1.
    \end{equation*}
At the same time for $t \in \bigr[ R \rho(y)/\threeess, 2 \rho(y) \bigl]$ we have
$t \, r[t] = At + 2(1-A)  \, \rho(y) \leq 2 A \rho(y)+ 2(1-A)  \, \rho(y) 
= 2 \rho(y)$. This completes the proof of \eqref{E:f.dil.is.dilation}.
 
By item \ref{I:F.is.cont} above, $F$ is continuous on $C[\Pf]$. But, by \eqref{E:f.in.terms.of.F}, 
and part \ref{I:Exp.alpha.homeom} of definition \ref{D:fibering.by.cones},  
$f_{dilate, R}$ is continuous on $\mcl{C}$. By \eqref{E:when.f.dilate.is.identity}, 
$f_{dilate, R}$ is the identity on $\D \setminus \mcl{C}$ and, by 
\eqref{E:F.from.f.dil}, \eqref{E:F.on.C[Pf].defn}, \eqref{E:f.in.terms.of.F}, and \eqref{E:Cs.defn}, $F$ is the identity on $C[\Pf] \setminus C_{2}$. It follows that 
    \begin{equation} \label{E:f.dilate.cont.onD}
      f_{dilate, R} \text{ is continuous on } \D .
    \end{equation} 

We \emph{claim} 
    \begin{equation}  \label{E:f.dil.is.G.equivar}
      f_{dilate,R} \text{ and $f_{dilate,R}^{-1}$ are $G$-equivariant} .
    \end{equation}
I.e., $f_{dilate,R}$ commutes with $g \in G$. By \eqref{E:fdil.is.identity.on.Pf}, this is trivial 
on $\D \setminus \mcl{B}_{\twoess}$, since, by \eqref{E:G.invariance.of.BS.CS}, we have that $\mcl{B}_{\twoess}$ is $G$-invariant. So suppose $x \in \mcl{B}_{\twoess}$ and 
let $(y, v) := \alpha(x) \in C_{2}$. Since, by \eqref{E:Bs.defn}, the restriction 
$\alpha \restriction_{\mcl{B}_{2}}$ is a bijection onto $C_{2}$, it suffices to show 
	\begin{equation}  \label{E:fg=gf}
		\alpha \circ f_{dilate,R} \circ g \circ \alpha^{-1}(y,v) 
		  = \alpha \circ g \circ f_{dilate,R} \circ \alpha^{-1}(y,v) ,
		    \quad (y,v) \in C_{2} .
	\end{equation}
By \eqref{E:G.invariance.of.BS.CS}, 
$g_{\ast}(y,v) = \bigl( g(y), v' \bigr) \in C_{2}$ for some $v' \in \RR^{k}$. 
By \eqref{E:f.dil.is.dilation}, 
$\alpha \circ f_{dilate,R} \circ \alpha^{-1} \bigl( g(y), v' \bigr) = \bigl( g(y), r(g(y),v') \, v' \bigr)$ for some $r(g(y),v') \geq 1$. Similarly, 
$\alpha \, \circ \, f_{dilate,R} \, \circ \, \alpha^{-1} ( y, v ) = \bigl( y, r(y,v) v \bigr)$ for 
some $r(y,v) \geq  1$. By \eqref{E:gExp=Expg.star}, 
	\begin{align*}
		\alpha \circ f_{dilate,R} \circ g \circ \alpha^{-1}(y,v) 
		     &= \alpha \circ f_{dilate,R} \circ \alpha^{-1} 
		       \circ (\alpha \circ g \circ \alpha^{-1})(y,v) \\
		     &= \alpha \circ f_{dilate,R} \circ \alpha^{-1} \circ g_{\ast}(y,v) \\
		     &= \bigl( g(y), r(g(y),v') \, v' \bigr) . 
	\end{align*}
Now, by \eqref{E:F.on.C[Pf].defn}, $r(y,v)$ only depends on $\rho(y)$ and $|v|$. By \eqref{E:rho.0.off.U.rho.g-invar}, $\rho \bigl( g(y) \bigr) = \rho(y)$. By\eqref{E:Riem.metrics.on.D.on.Rk.same} and the fact that
$g^{\ast} \bigl( \langle \cdot, \cdot \rangle \bigr) = \langle \cdot, \cdot \rangle$ (by assumption), 
    \begin{equation*}
      |v'| = \bigl| (g(y), v') \bigr| =  \bigl| g_{\ast}(y,v) \bigr| = \bigl| (y,v) \bigr| = |v| .
    \end{equation*} 
Therefore, $r \bigl( g(y),v' \bigr) = r(y,v)$. Thus, by \eqref{E:alpha.is.G-equivar} again and definition of $v'$, 
	\begin{align*}
		\alpha \circ f_{dilate,R} \circ g \circ \alpha^{-1}(y,v) &= \bigl( g(y), r(|v|) v' \bigr) 
		  = r(|v|) \bigl( g(y), v' \bigr) \\
		  &= r(|v|) g_{\ast}(y, v) = g_{\ast} \bigl( y, r(|v|) v \bigr) \\
		  &= g_{\ast} \circ \alpha \circ f_{dilate,R} \circ \alpha^{-1} (y, v) \\
		  &= \alpha \circ g \circ f_{dilate,R} \circ \alpha^{-1} (y, v),
	\end{align*}
(recall \eqref{E:vector.ops.on.TD}) which is just \eqref{E:fg=gf}. Equivariance of 
$f_{dilate,R}^{-1}$ follows immediately from that of $f_{dilate,R}$. This proves the claim.

Note that, by \eqref{E:f.in.terms.of.F}, \eqref{E:Bs.defn}, and \eqref{E:F.on.parts} of $F$, 
	\begin{multline}  \label{E:f.dil.on.B's}
		f_{dilate,R} \text{ maps } \mcl{B}_{R/\threeess} 
		  \subset \mcl{B}_{\ess} \text{ injectively onto } \mcl{B}_{\ess} \\
		    \text{ and maps } \mcl{B}_{\twoess} \setminus \mcl{B}_{R/\threeess} 
		      \text{ injectively onto } 
		      \mcl{B}_{\twoess} \setminus \mcl{B}_{\ess} \\
		       \text{ and } \mcl{C} \setminus \mcl{B}_{\twoess} 
		         \text{ injectively onto } \mcl{C} \setminus \mcl{B}_{\twoess}.
	\end{multline} 
In particular, by \eqref{E:F.is.injective}, we have that $f_{dilate,R}$ maps $\mcl{C}$ injectively onto $\mcl{C}$. It follows that $f_{dilate,R}$ is a bijection of $\D$ onto itself. 
So $f_{dilate,R}$ has an inverse and $f_{dilate,R}^{-1}$ maps $\mcl{C}$ injectively 
onto $\mcl{C}$. Since $f_{dilate,R}$ is $G$-equivariant, to so is $f_{dilate,R}^{-1}$.

Recall \eqref{E:fdil.is.bijective}. See appendix \ref{Chptr:misc.proofs} for the proof of the following. 
  \begin{lemma}  \label{L:f.dil.Lipschitz}
For $R \in (0, \threeess/3)$ (see \eqref{E:R.<.threeess/3.A.B.ineqs}) fixed, $f_{dilate,R}$ and its inverse are Lipschitz on $\D$ w.r.t.\ $\xi$. The Lipschitz constant for $f_{dilate,R}$ is inversely proportional to $R$ in the sense that there exists 
$K < \infty$ depending only on $\D$ and $C[\Pf]$ s.t.\ $K/R$ is a Lipschitz constant for  $f_{dilate,R}$ for every $R \in (0, \threeess/3)$. 
  \end{lemma}
In particular, $f_{dilate, R}^{-1}$ is continuous, so 
    \begin{equation}  \label{E:fdil.is.homeom}
      f_{dilate, R} \text{ is a homeomorphism. }
    \end{equation}
Therefore, by \eqref{E:g.commutes.w/.set.ops}, if we let 
    \begin{equation}  \label{E:Phi.dilate.defn}
      \Phi_{R_{m}, dilate} := \Phi_{R_{m}} \circ f_{dilate,R_{m}}^{-1}
    \end{equation} 
then $\Phi_{R_{m}, dilate}$ is continuous off 
	\begin{equation}   \label{E:S.dilate.defn}
		\Ss_{m,dilate}' := f_{dilate,R_{m}} (\Ss_{m}').
	\end{equation}
Thus, $\Phi_{R_{m}, dilate}$ is continuous 
on $\D_{m,dilate}' := \D \setminus \Ss_{m,dilate}'$. 

Since $f_{dilate,R_{m}}$ is Lipschitz, by lemma \ref{L:f.dil.Lipschitz}, we have, by \eqref{E:Sm'.countably.rect} 
and Hardt and Simon \cite[Definition 2.1, p.\ 20]{rHlS86.GMT} that 
    \begin{equation*}  \label{E:Sm.dil.is.countably.rect}
      \Ss_{m,dilate}' \text{ is countably $a$-rectifiable.}
    \end{equation*} 

\emph{Claim:}
    \begin{equation*}  \label{E:Phi.dil.etc.has.prpty}
      (\Phi_{R_{m}, dilate}, \Ss_{m,dilate}', G, \T, a) 
        \text{ satisfies property \ref{Pty:agree.near.T}.}
    \end{equation*}
Via part \ref{I:pty:agree.near.T.holds.forPhi.m} of property \ref{Pty:four.properties}, 
of $\{ \Phi_{m}, R_{m} \}$ in subsection \ref{SS:Small.R.rect} 
we know that $(\Phi_{m}, \Ss_{m}', G, \T, a)$ satisfies property \ref{Pty:agree.near.T}. Thus, 
$\Ss_{m}'$ is closed with empty interior. But $f_{dilate, R}$ is a homeomorphism. Therefore, 
$\Ss_{m,dilate}'$ is also closed with empty interior. 
By \eqref{E:f.dil.is.G.equivar}, we have that $\Phi_{R_{m}, dilate}$ and 
$\Ss_{m,dilate}'$ are $G$-invariant. We know that $\Phi_{R_{m}, dilate}$ is continuous 
on $\D_{m,dilate}' := \D \setminus \Ss_{m,dilate}'$. 

By \eqref{E:fdil.is.identity.on.Pf}, \eqref{E:g.commutes.w/.set.ops}, and \eqref{E:S.dilate.defn}, 
    \begin{multline*}
      \Ss_{m}' \cap \T =  f_{dilate,R_{m}}(\Ss_{m}' \cap \T) 
        = f_{dilate,R_{m}}(\Ss_{m}') \cap f_{dilate,R_{m}}(\T) \\ 
          = f_{dilate,R_{m}}(\Ss_{m}') \cap \T = \Ss_{m,dilate}' \cap \T .
    \end{multline*}
I.e. $\Ss_{m}' \cap \T = \Ss_{m,dilate}' \cap \T$. Similarly, by \eqref{E:fdil.is.identity.on.Pf}, 
if $x \in \T$ then $\Phi_{R_{m}, dilate}(x) = \Phi_{R_{m}} \circ f_{dilate,R_{m}}^{-1}(x) 
= \Phi_{R_{m}}(x)$. I.e., $\Phi_{R_{m}, dilate}$ and $\Phi_{R_{m}}(x)$ agree on $\T$. 
We conclude that \linebreak 
$(\Phi_{R_{m}, dilate}, \Ss_{m,dilate}', G, \T, a)$ inherits property \ref{Pty:agree.near.T} from $(\Phi_{m}', \Ss_{m}', G, \T, a)$. This proves claim that 
$(\Phi_{R_{m}, dilate}, \Ss_{m,dilate}', G, \T, a)$ satisfies property \ref{Pty:agree.near.T}.

Hence, $\Hm^{a}(\Ss_{m,dilate}') > 0$. In particular, by \eqref{E:Haus.dim.defn}, 
$\dim \Ss_{m,dilate}' \geq a$. (Alternatively, proceed as follows. 
Let $K < \infty$ be a Lipschitz constant for $f_{dilate, R_{m}}^{-1}$. By part \ref{I:pty:agree.near.T.holds.forPhi.m} of property \ref{Pty:four.properties}, and thence property \ref{Pty:agree.near.T}, we have $\Hm^{a}(\Ss_{m}') > 0$. Therefore, by \eqref{E:Lip.magnification.of.Hm}, 
$0 < \Hm^{a}(\Ss_{m}') = \Hm^{a} \bigl[  f_{dilate, R}^{-1}(\Ss_{m,dilate}') \bigr] 
\leq K^{a} \Hm^{a} (\Ss_{m,dilate}')$.)

On the other hand, by \eqref{E:for.all.m.dim.Sm.=.a} and lemma \ref{L:loc.Lip.image.of.null.set.is.null}, we have 
$\dim \Ss_{m,dilate}' \leq \dim \Ss_{m}' = a$. I.e.,
	\begin{equation*}
		\dim \Ss_{m,dilate}' = a.
	\end{equation*}

Recall \eqref{E:S.dilate.defn}. By \eqref{E:Hma.S'.cap.B.Rm.0} and the fact (see \eqref{E:R<threeess}) that 
$\threeess \geq 1$, we have  
    \begin{equation}  \label{E:Hma.S'.meet.B.R/threeess.0}
        \Hm^{a} (\Ss_{m}'  \cap \mcl{B}_{R_{m}/\threeess}) = 0.
    \end{equation}
By \eqref{E:f.dil.on.B's} and \eqref{E:g.commutes.w/.set.ops}, we have that 
$f_{dilate,R_{m}} : (\mcl{B}_{R_{m}/\threeess})^{c} \to (\mcl{B}_{\ess})^{c}$. 
(``${}^{c}$'' indicates complementation w.r.t.\ $\D$). By lemma \ref{L:f.dil.Lipschitz}, $f_{dilate,R_{m}}$ is Lipschitz. That means, first, by \eqref{E:Ha.Sm'.finite} and \eqref{E:Lip.magnification.of.Hm}, $\Hm^{a} (\Ss_{m, dilate}') < \infty$. Second, by \eqref{E:Hma.S'.meet.B.R/threeess.0} and lemma \ref{L:loc.Lip.image.of.null.set.is.null}, 
$\Hm^{a}( \Ss_{m, dilate}' \cap \mcl{B}_{\ess} ) = 0$. Summing up:
	\begin{equation}  \label{E:Ha.and.S.dilate} 
		\Hm^{a} (\Ss_{m, dilate}') < \infty 
		  \text{ and } \Hm^{a}( \Ss_{m, dilate}' \cap \mcl{B}_{\ess} ) = 0.
	\end{equation}

Let $R_{0}$ be as in \eqref{E:R0.choice}. Then, by \eqref{E:Bs.w/in.s.of.T}, $R_{0} \leq 1/2 < 1$. Thus, \eqref{E:outside.B.toT.>.R0} and \eqref{E:Hma.S'.cap.B.eta.R0.0} hold with $\Ss' = \Ss_{m, dilate}'$. 
Therefore, from subsection \ref{SS:R.geq.R0},
	\begin{equation}  \label{E:lwr.bnd.on.Ha.S.dilate}
		\Hm^{a}( \Ss_{m, dilate}') \geq \Omega 
		  := \gamma(R_{0}) R_{0}^{\min( d-p-1, a)} > 0.
	\end{equation}
($\gamma(R_{0}) > 0$ is defined in \eqref{E:gamma.R0.defn}.) Note that $\Omega$ does not depend on $(\Phi_{m}, \Ss_{m}')$.

Now, by \eqref{E:f.dil.on.B's}, $f_{dilate, R_{m}}$ maps 
$\mcl{B}_{\twoess}$ injectively onto $\mcl{B}_{\twoess}$ and 
$\mcl{C} \setminus \mcl{B}_{\twoess}$ injectively 
onto $\mcl{C} \setminus \mcl{B}_{\twoess}$. Therefore, by \eqref{E:g.commutes.w/.set.ops}
 \eqref{E:S.dilate.defn} and the fact that $f_{dilate,R_{m}}$ is the identity 
off $\mcl{B}_{2}$, we have 
	\begin{multline}  \label{E:S.dilate.S.agree.off.B.twoess}
	  \Ss_{m,dilate}' \setminus \mcl{B}_{\twoess} 
	    = f_{dilate,R_{m}} (\Ss_{m}') \setminus \mcl{B}_{\twoess} 
	      = f_{dilate,R_{m}} (\Ss_{m}') \setminus f_{dilate,R_{m}} (\mcl{B}_{\twoess}) \\
	        = f_{dilate,R_{m}} (\Ss_{m}' \setminus \mcl{B}_{\twoess}) 
		  = \Ss_{m}'  \setminus \mcl{B}_{\twoess}.
	\end{multline}
Hence, by \eqref{E:Ha.Sm'.finite}, 
	\begin{equation*}
		\Hm^{a}(\Ss_{m,dilate}' \setminus \mcl{B}_{\twoess}) 
		  =  \Hm^{a}(\Ss_{m}' \setminus \mcl{B}_{\twoess})
	            \leq  \Hm^{a}(\Ss_{m}') < \infty .
	\end{equation*}
But by \eqref{E:Rm.bounded}, eventually
	\begin{equation} \label{E:few.dilate.sings.outside.B.2S}
		\Hm^{a}(\Ss_{m,dilate}' \setminus \mcl{B}_{\twoess}) 
		  =  \Hm^{a}(\Ss_{m}' \setminus \mcl{B}_{\twoess}) < \Omega/2.
	\end{equation} 
Thus, by \eqref{E:lwr.bnd.on.Ha.S.dilate} and \eqref{E:few.dilate.sings.outside.B.2S}, we have 
	\begin{multline} \label{E:S.dllate.inside.outside.B}
		\Omega \leq \Hm^{a}(\Ss_{m,dilate}') 
		  \leq \Hm^{a}(\Ss_{m,dilate}'  \setminus \mcl{B}_{\twoess})
		    + \Hm^{a}( \Ss_{m,dilate}'  \cap \mcl{B}_{\twoess}) \\
		      < \Omega/2 + \Hm^{a}( \Ss_{m,dilate}'  \cap \mcl{B}_{\twoess}).
	\end{multline}
	
	The following is proved in the next subsection.
  \begin{lemma}  \label{L:dilate.R.-(d-p-1)}
    Let $\{ \Phi_{m}, \Ss_{m}', R_{m} \}$ have properties \ref{I:sing.set.in.Sm'} through \ref{I:Hm.Sm.to.0.fast} in section 
    \ref{SS:Small.R.rect}. There is a constant $K_{3} < \infty$ depending only 
    on $\D$, $C[\Pf]$, $\T$, and $a$ s.t.\, 
    if $0 < R_{m} < \threeess/3$, then 
	\begin{equation}  \label{E:dilate.R.-(d-p-1)}
		\Hm^{a}(\Ss_{m,dilate}' \cap \mcl{B}_{\twoess}) 
		   \leq K_{3} R_{m}^{-\min( d-p-1, a)} \Hm^{a}(\Ss_{m}').
	\end{equation}
  \end{lemma}
Therefore, by \eqref{E:S.dllate.inside.outside.B}, \eqref{E:dilate.R.-(d-p-1)}, and part \ref{I:Hm.Sm.to.0.fast} of property \ref{Pty:four.properties}, 
	\begin{equation*}
	      0 < \Omega/2 < \Hm^{a}(\Ss_{m,dilate} \cap \mcl{B}_{\twoess})  
	          \leq K_{3} R_{m}^{-\min( d-p-1, a)} \, \Hm^{a} ( \Ss_{m}'  ) \to 0 
	            \text{ as } m \to \infty .
	\end{equation*}
Contradiction. Thus, a sequence  of $\{ \Phi_{m} \}$ having properties \ref{I:sing.set.in.Sm'}, \ref{I:pty:agree.near.T.holds.forPhi.m}, \ref{I:Rm.away.from.Pf}, and \ref{I:Hm.Sm.to.0.fast} of subsection \ref{SS:Small.R.rect} cannot exist. 
Therefore, to complete the proof of the theorem, we ``just'' need to prove lemma \ref{L:dilate.R.-(d-p-1)}!
  
\subsection{Proof of lemma \ref{L:dilate.R.-(d-p-1)}} \label{SS:Pf.of.big.lemma}
First, suppose $a = 0$. By \eqref{E:0.dim.Haus.measure}, $\Hm^{a}$ is just cardinality. Since $f_{dilate, R_{m}}$ is a bijection (see \eqref{E:f.dil.on.B's}), by \eqref{E:S.dilate.defn}, we have 
    \begin{equation*}
         \Hm^{0}(\Ss_{m,dilate}' \cap \mcl{B}_{\twoess}) 
           = \Hm^{0}(\Ss_{m}' \cap \mcl{B}_{\twoess})
             \leq R_{m}^{-0} \Hm^{a}(\Ss_{m}') .
    \end{equation*}
I.e., \eqref{E:dilate.R.-(d-p-1)} holds. So from now on we assume $a > 0$. Since $a$ is an integer, by \eqref{E:Sm'.countably.rect}, we thus assume, 
    \begin{equation*}
      a \geq 1 .
    \end{equation*}

Now suppose $0 < a \leq d-p-1$. By \eqref{E:Rm.=.dist.Sm.P} and lemma \ref{L:f.dil.Lipschitz} there exists $K_{3} < \infty$ depending only on $\D$, $C[\Pf]$, $\T$, and $a$ (in particular, $K_{3}$ does not depend on $m$) s.t.\ $K_{3}^{1/a} / R_{m}$ is a Lipschitz constant for $f_{dilate,R_{m}}$. Therefore, by \eqref{E:g.commutes.w/.set.ops}, \eqref{E:S.dilate.defn}, \eqref{E:f.dil.on.B's}, and \eqref{E:Lip.magnification.of.Hm}, 
	\begin{align*}  \label{E:dilate.R.-(d-p-1)}
		\Hm^{a}(\Ss_{m,dilate}' \cap \mcl{B}_{\twoess}) 
		  &= \Hm^{a}\bigl[ f_{dilate, R_{m}} (\Ss_{m}' \cap \mcl{B}_{\twoess}) \bigr] \\
		  &\leq K_{3} R_{m}^{-a} \Hm^{a}(\Ss_{m}' \cap \mcl{B}_{\twoess}) \\
		  &\leq K_{3} R_{m}^{-a} \Hm^{a}(\Ss_{m}') \\
		  &= K_{3} R_{m}^{-\min( d-p-1, a)} \Hm^{a}(\Ss_{m}') .
	\end{align*} 
I.e., \eqref{E:dilate.R.-(d-p-1)} again holds. 

Therefore, from now on we may assume $a > d-p-1$. By assumption $a < d$. Moreover, by \eqref{E:Sm'.countably.rect}, $a$ is an integer. Therefore, $a \leq d-1$. Hence, if $p = 0$, then $a \leq d-p-1$. In summary, we may assume,
    \begin{equation} \label{E:p.poz.a.big}
        p > 0 \text{ and } a > d-p-1.
    \end{equation}
Thus, $-\min( d-p-1, a) = p-d+1$. \eqref{E:dilate.R.-(d-p-1)} holds if an only if it does with $R_{m}$ replaced by $R_{m}/2$. Making this substitution means that part \ref{I:Rm.away.from.Pf} of property \ref{Pty:four.properties} still holds with 
    \begin{equation}  \label{E:dist.>.Rm}
      dist^{a}(\Ss_{m}', \Pf) > R_{m} > 0 ,
    \end{equation}
i.e., with two strict inequalities.

Recall \eqref{E:boldF.[E].defns} and \eqref{E:overline.U.defn}. Let 
    \begin{equation}  \label{E:Ym.defn}
      \mcl{Y} := \mcl{Y}_{m} := Exp \bigl( \mbf{F}_{[0,R_{m}]}[\overline{\clU}] \bigr)
    \end{equation}
Let $x \in \mcl{Y}$. There exists $(y, v) \in \mbf{F}_{[0,R_{m}]}$ s.t.\ $Exp(y,v) = x$. Then, by \eqref{E:dist.Exp.v.to.P.bnd}, 
$R_{m} \geq |v| \geq \text{dist}(x, \Pf)$. Hence, by \eqref{E:dist.>.Rm}, we have
	\begin{equation} \label{E:Ha.Sm.cap.Ym.0}
	    \Hm^{a} ( \Ss_{m}' \cap \mcl{Y}_{m} ) = 0. 
	\end{equation}
Therefore, since $f_{dilate, R_{m}} : \D \to \D$ is a bijection, by \eqref{E:S.dilate.defn}, \eqref{E:g.commutes.w/.set.ops}, lemma \ref{L:loc.Lip.image.of.null.set.is.null}, and lemma \ref{L:f.dil.Lipschitz}, we have 
	\begin{equation*}
	    \Hm^{a} \bigl[ \Ss_{m,dilate}' \cap f_{dilate, R_{m}} (\mcl{Y}_{m}) \bigr] 
	      = \Hm^{a} \bigl[ f_{dilate, R_{m}} ( \Ss_{m}' \cap \mcl{Y}_{m} ) \bigr] = 0. 
	\end{equation*}
Hence, by \eqref{E:f.dil.on.B's}, \eqref{E:g.commutes.w/.set.ops}, and the fact that $f_{dilate, R_{m}}$ is a bijection,
	\begin{multline*}
		\Hm^{a} ( \Ss_{m,dilate}' \cap \mcl{B}_{\twoess} ) 
		  \leq \Hm^{a} \Bigl[ \Ss_{m,dilate}' \cap \bigl( \mcl{B}_{\twoess} 
		    \setminus f_{dilate, R_{m}} (\mcl{Y}_{m}) \bigr) \Bigr] 
		    	+ \Hm^{a} \bigl[ \Ss_{m,dilate}' 
			  \cap f_{dilate, R_{m}} (\mcl{Y}_{m}) \bigr] \\
		      = \Hm^{a} \Bigl[ \Ss_{m,dilate}' \cap \bigl( \mcl{B}_{\twoess} 
		        \setminus f_{dilate, R_{m}} (\mcl{Y}_{m}) \bigr) \Bigr]
		          = \Hm^{a} \Bigl( f_{dilate, R_{m}} \bigl[ \Ss_{m}' 
		             \cap ( \mcl{B}_{\twoess} \setminus \mcl{Y}_{m} ) \bigr] \Bigr).
	\end{multline*}

Therefore, by \eqref{E:Ha.Sm.cap.Ym.0} and \eqref{E:p.poz.a.big}, \emph{it suffices to show}   
	\begin{equation}  \label{E:f.dil.ineq.Ym}
		 \Hm^{a} \Bigl( f_{dilate, R_{m}} \bigl[ \Ss_{m}' 
				    \cap ( \mcl{B}_{\twoess} \setminus \mcl{Y}_{m} ) \bigr] \Bigr)
		    \leq K_{3} R_{m}^{p-d+1} \Hm^{a} \bigl( \Ss_{m}' 
				    \cap ( \mcl{B}_{\twoess} \setminus \mcl{Y}_{m} ) \bigr),
	\end{equation}
for some $K_{3} < \infty$ depending only on $\D$, $C[\Pf]$, $\T$, and $a$.

\subsubsection{Change of variables} \label{SSS:change.of.vars} 
We simplify \eqref{E:f.dil.ineq.Ym} through repeated change of variables. Through change of variables and restriction, we transform $f_{dilate,R_{m}}$ into a function to which we can conveniently apply calculus. Through several steps we translate the problem into one concerning convex subsets of Euclidean space. 

Let $X$ be a metric space, let $S \subset X$, and let $f : X \to X$. Let $Y$ be another metric spaces, let $\phi : X \to Y$ be a bi-Lipschitz bijection, and let $L < \infty$ be a Lipschitz constant for both $\phi$ and $\phi^{-1}$. Assume $L$ depends only on $\D$, $C[\Pf]$, $\T$, and $a$. Then by \eqref{E:Lip.magnification.of.Hm}, 
	\begin{equation}  \label{E:2.Hma.Lip.ineqs}
		\Hm^{a} \bigl[ f(S) \bigr] = \Hm^{a} \Bigl( \phi^{-1} \bigl[ \phi \circ f(S) \bigr] \Bigr) 
		  \leq L^{a} \Hm^{a} \bigl[ \phi \circ f(S) \bigr] \text{ and }
		    \Hm^{a} \bigl[ \phi(S) \bigr] \leq L^{a} \Hm^{a} (S),
	\end{equation}
where in each case $\Hm^{a}$ is computed using Hausdorff measure based on the appropriate metric. Note that  the first inequality in the preceding can be written 
    \begin{equation}  \label{E:first.ineq.rewrite}
        \Hm^{a} \bigl[ f(S) \bigr] 
          \leq L^{a} \Hm^{a} \Bigl( (\phi \circ f \circ \phi^{-1}) \bigl[ \phi(S) \bigr] \Bigr). 
    \end{equation}
Suppose for some $K < \infty$ we have
    \begin{equation*}
      \Hm^{a} \Bigl( (\phi \circ f \circ \phi^{-1}) \bigl[ \phi(S) \bigr] \Bigr) 
    		  \leq K R_{m}^{p-d+1} \, \Hm^{a} \bigl[ \phi(S) \bigr] .
    \end{equation*}
Then, by \eqref{E:first.ineq.rewrite} and the second inequality in \eqref{E:2.Hma.Lip.ineqs},
    \begin{equation*}
      \Hm^{a} \bigl[ f(S) \bigr] 
        \leq L^{a} \Hm^{a} \Bigl( (\phi \circ f \circ \phi^{-1}) \bigl[ \phi(S) \bigr] \Bigr)
          \leq L^{a} K R_{m}^{p-d+1} \, \Hm^{a} \bigl[ \phi(S) \bigr]
            \leq (L^{2a} K) R_{m}^{p-d+1} \, \Hm^{a} (S) .
    \end{equation*}
Summing up, we get
	\begin{gather} \label{E:general.change.of.variables}
		\text{\emph{IF} } 
		  \Hm^{a} \Bigl( (\phi \circ f \circ \phi^{-1}) \bigl[ \phi(S) \bigr] \Bigr) 
		  \leq K R_{m}^{p-d+1} \, \Hm^{a} \bigl[ \phi(S) \bigr]
	            \text{ for some } K < \infty \\ 
	          \text{ \emph{THEN} }  
		\Hm^{a} \bigl[ f(S) \bigr] \leq (L^{2a} K) R_{m}^{p-d+1} \, \Hm^{a} (S). \notag 
	\end{gather}
Suppose $K$ in \eqref{E:general.change.of.variables} depends only on $\D$, $C[\Pf]$, 
$\T$, and $a$. Then for any relevant $S$ and $f$, we can work in $Y$ instead of $X$ and with $\phi \circ f \circ \phi^{-1}$ instead of $f$. Note that, since $\phi$ is bi-Lipschitz, by \eqref{E:Lip.magnification.of.Hm}, we have that $\phi(S)$ has finite $\Hm^{a}$ measure if and only if $S$ does.

\emph{First application of} \eqref{E:general.change.of.variables}: The first step is use \eqref{E:general.change.of.variables} to change the problem about events in $\mcl{B}_{2}$ to one concerning events in $C[\clU]$. Recall the definition 
of $\mcl{C}$ in part \ref{I:Exp.alpha.homeom} of definition \ref{D:fibering.by.cones}. Recall also the definition, \eqref{E:boldF.[E].defns}, 
of $\mbf{F}$. Using the fact that $Exp = \alpha^{-1}$ and $\alpha$ are bijections to and from $\mcl{C}$, \eqref{E:g.commutes.w/.set.ops}, \eqref{E:Bs.defn},  \eqref{E:C.sub.s.in.Bold.F.[0,s)}, and \eqref{E:Ym.defn}, let 
	\begin{multline}  \label{E:S.2.defn}
		\mathscr{S}_{2} := \mathscr{S}_{2,m} := \alpha(\Ss_{m}' \cap \mcl{C}) 
		  \cap \bigl( C_{\twoess} \setminus \mbf{F}_{[0,R_{m}]}[\overline{\clU}] \bigr) \\
		    = \alpha \Bigl[ (\Ss_{m}' \cap \mcl{C}) \cap (\mcl{B}_{2} 
		      \setminus \mcl{Y}_{m}) \Bigr]  
		        = \alpha \Bigl[ \Ss_{m}' \cap (\mcl{B}_{2} \setminus \mcl{Y}_{m}) \Bigr] 
		          \subset C[\clU] \subset C[\overline{\clU}].
	\end{multline}
Notice that, since $\mathscr{S}_{2} \subset C_{\twoess} \setminus \mbf{F}_{[0,R_{m}]}$, we have
    \begin{equation}  \label{E:pts.of.S2.are.Rm.away.from.Pf}
        (y, v) \in \mathscr{S}_{2} \text{ implies } |v| > R_{m} .
    \end{equation}
By \eqref{E:g.commutes.w/.set.ops} and \eqref{E:general.change.of.variables} with 
$\phi = \alpha$ and $S  = \Ss_{m}' \cap (\mcl{B}_{2} \setminus \mcl{Y}_{m})$, to prove \eqref{E:f.dil.ineq.Ym}, and hence the lemma, \emph{it suffices to show}
	\begin{equation}   \label{E:dilate.alpha.ineq}
	     \Hm^{a} \bigl[ ( \alpha \circ f_{dilate, R_{m}} 
	       \circ \alpha^{-1} ) ( \mathscr{S}_{2,m} ) \bigr] 
		\leq K_{4} R_{m}^{p-d+1} \, 
		  \Hm^{a} \bigl( \mathscr{S}_{2,m} \bigr)
	\end{equation}
for some $K_{4} < \infty$ depending only on $\D$, $C[\Pf]$, $\T$, and $a$. (Here, we compute $\Hm^{a}$ based 
on the metric $\xi_{+}$; see \eqref{E:xi+.from.2.metrics}.) 

By \eqref{E:W.U.T.subset} and \eqref{E:overline.U.defn}, $\overline{\clU}$ is compact. 
By \eqref{E:C.sub.s.in.Bold.F.[0,s)}, $C_{\twoess} \subset C[\overline{\clU}]$. By \eqref{E:Bs.defn}, we have $\mcl{B}_{\twoess} = Exp(C_{\twoess})$. Therefore, by part \ref{I:Exp.alpha.homeom} of definition \ref{D:fibering.by.cones}, we have that 
 $\alpha : \mcl{C} \to C[\Pf]$ and $\alpha^{-1} = Exp \restriction_{C[\Pf]}$ are Lipschitz on 
$\mcl{B}_{\twoess}$ and $C_{\twoess}$, resp. 
Now, by part \ref{I:sing.set.in.Sm'} of property \ref{Pty:four.properties}, 
$\Ss_{m}'$ is closed, hence Borel. Since $\mcl{C}$ is open (by part \ref{I:Exp.alpha.homeom} of definition \ref{D:fibering.by.cones}) we have that 
$\Ss_{m}' \cap \mcl{C}$ is Borel. Therefore, $\alpha(\Ss_{m}' \cap \mcl{C}) = Exp^{-1}(\Ss_{m}' \cap \mcl{C})$ is Borel. Moreover, $C_{\twoess} \setminus \mbf{F}_{[0,R_{m}]}$ is Borel. By now applying \eqref{E:Sm'.countably.rect} and \eqref{E:Lip.magnification.of.Hm}, we now conclude
	\begin{equation} \label{E:S2.Borel.finite.meas}
		\mathscr{S}_{2} \text{ is Borel, countably $a$-rectifiable, and } 
		  \Hm^{a}(\mathscr{S}_{2}) < \infty.
	\end{equation} 
This finiteness of $\Hm^{a}$-measure property will be preserved under all further manipulations of $\Ss_{m}'$ performed below.

Now, by \eqref{E:C.sub.s.in.Bold.F.[0,s)} and \eqref{E:S.2.defn}, we have 
$\pi(\mathscr{S}_{2}) \subset \clU \subset \overline{\clU}$ (see \eqref{E:pi.is.proj}) and, 
by definition \ref {D:fibering.by.cones} part \ref {I:local.triv} and \eqref{E:overline.U.defn}, 
$\overline{\clU}  \subset \Pf$ is covered by finitely many sets $\mcl{V}$ with properties described there. Let $\mcl{V}$ be a generic one. Recall \eqref{E:CL.in.Euc.space}. 
$\mcl{V}$, in turn, is covered by finitely many sets $\mcl{A}_{i}$ with associated cone 
$\msf{CL}_{i}$ and injective map 
$h_{i} : \mcl{A}_{i} \times \msf{CL} _{i}\to C[\mcl{A}_{i}]$. Thus, 
    \begin{equation}  \label{E:clU.bar.convered.by.finitely.many.Ai's}
      \overline{\clU} \text{ is covered by finitely many sets } \mcl{A}_{i} .
    \end{equation}

Since each of the $\msf{L}_{i}$'s is a compact subset of some Euclidean space (see \eqref{E:CL.in.Euc.space}), and there are only finitely many of them, there is a constant, 
$\mathscr{L} \in (0, \infty)$, that does not depend on $\bigl\{ (\Phi_{m}, \Ss_{m}') \bigr\}$ but just on $\D$, $C[\Pf]$, and $\T$, s.t.\ if $z \in \msf{L}_{i}$ then $|z| < \mathscr{L}$. In fact, we can take $\mathscr{L} = 2$: Just replace 
$\msf{L}_{i}$ by $\msf{L}_{i}' := (2/\mathscr{L}) \msf{L}_{i} \subset \RR^{J}$ (for appropriate $J$) and $h_{i}$ by 
$h_{i}' : \bigr( x, s(1, z') \bigl) := h_{i} \Bigr( x, s \bigl(1, (\mathscr{L}/2) z' \bigr) \Bigl)$ 
($x \in \mcl{A}_{i}$, 
$s \in [0,1)$, and $z' \in \msf{L}_{i}'$; \eqref{E:CL.in.Euc.space} again). Thus, we may assume
	\begin{equation}  \label{E:msfZ.unif.bounded}
		\text{If } z \in \msf{L}_{i} \text{ then } |z| < 2. 
	\end{equation}

For $I \subset [0,1)$, define 
    \begin{equation*}
        I \cdot \bigl( \{1\} \times \msf{L}_{i} \bigr) 
          := \Bigl\{ r \cdot  \bigl( \{1\} \times \msf{L}_{i} \bigr) : r \in I  \Bigr\}.
    \end{equation*}
Here ``$\cdot$'' means scalar multiplication. 

Since each of the $\msf{L}_{i}$'s corresponding to an $\mcl{A}_{i}$ in \eqref{E:clU.bar.convered.by.finitely.many.Ai's} is a compact subset of some Euclidean space (see \eqref{E:CL.in.Euc.space}), and there are only finitely many of them, there is a constant, $\mathscr{L} \in (0, \infty)$, that does not depend on 
$\bigl\{ (\Phi_{m}, \Ss_{m}') \bigr\}$ but just 
on $\D$, $C[\Pf]$, and $\T$, s.t.\ if $z \in \msf{L}_{i}$ then $|z| < \mathscr{L}$. In fact, we can take $\mathscr{L} = 2$: 
Just replace $\msf{L}_{i}$ by $\msf{L}_{i}' := (2/\mathscr{L}) \msf{L}_{i} \subset \RR^{J}$ (for appropriate $J$) and $h_{i}$ by 
$h_{i}' : \bigr( x, s(1, z') \bigl) := h_{i} \Bigr( x, s \bigl(1, (\mathscr{L}/2) z' \bigr) \Bigl)$ 
($x \in \mcl{A}_{i}$, 
$s \in [0,1)$, and $z' \in \msf{L}_{i}'$; \eqref{E:CL.in.Euc.space} again). Thus, we may assume
	\begin{equation}  \label{E:msfZ.unif.bounded}
		\text{If } z \in \msf{L}_{i} \text{ then } |z| < 2. 
	\end{equation}

Let $c >0$ be an integer and suppose $G_{1}, G_{2} : X \to \RR^{c} \setminus \{ 0 \}$, where $X$ is a subspace of a metric space, $Y$. E.g., $Y$ could be the extended reals and $X$ the integers. Let $x_{0} \in \overline{X}$. If $\bigl| G_{1}(x) \bigr|/\bigl| G_{2}(x) \bigr|$ stays bounded and bounded away from 0 as $x \to x_{0}$ through $X$ we say that ``$G_{1}(x)$ is asymptotic to $G_{2}(x)$ as $x \to x_{0}$'' and write 
    \begin{equation}  \label{E:asymp}
         ``G_{1}(x) \asymp G_{2}(x) \text{ as } x \to x_{0}.''
    \end{equation} 
E.g., ``$G_{1}(x) \asymp 1$'' means $G_{1}(x)$ stays bounded and bounded away from 0 as $x \to x_{0}$. 

Recall \eqref{E:eps.=.2}. \emph{Claim:} There exists 
    \begin{equation}  \label{E:theta.m.in.(0,1)}
      \theta_{m} \in (0,1) 
    \end{equation} 
s.t.\ for all $\mcl{A}_{i}$ covering $\overline{\clU}$ we have
  \begin{equation}  \label{E:hi.invrs.theta.m}
	h_{i}^{-1} \bigl( \mbf{F}_{(R_{m},2)}[\overline{\clU}] \bigr)
	    \subset \mcl{A}_{i} \times \Bigl[ (\theta_{m}, 1) 
	      \cdot \bigl( \{1\} \times \msf{L}_{i} \bigr) \Bigr]
		\text{ and } \theta_{m} \asymp R_{m}, \text{ as } m \to \infty.                  
   \end{equation}
(Here $(\theta_{m}, 1)$ is the open interval from $\theta_{m}$ to $1$. 
See \eqref{E:boldF.[E].defns}.) Moreover, there are positive bounds, above and below, on the ratio $\theta_{m} / R_{m}$ that only depend 
on $(\D, C[\Pf], \T, a)$. Similar statements will apply in all instances of ``$\asymp$'' and ``$O(\cdot)$'' -- Landau ``big O'', de Bruijn \cite[Section 1.2]{ngdeB81.AsympMthdsAnlys}) -- below. (By \eqref{E:Rm.=.dist.Sm.P}, $ \theta_{m} \asymp R_{m}$ does not contradict \eqref{E:theta.m.in.(0,1)}.)

Before we prove \eqref{E:hi.invrs.theta.m}, observe the following. By \eqref{E:S.2.defn}, \eqref{E:g.commutes.w/.set.ops}, 
the fact that $\alpha = \bigl( Exp \restriction_{C[\Pf]} \bigr)^{-1}$, and \eqref{E:Ym.defn}, 
we have,  
  \begin{align}  \label{E:hi.invrs.S2.eps.re-expression}
	h_{i}^{-1} (\mathscr{S}_{2,m}) 
	  &= h_{i}^{-1} \left[ \alpha \Bigl[ \Ss_{m}' \cap (\mcl{B}_{2} 
	    \setminus \mcl{Y}_{m}) \Bigr]  \right] \notag \\
	  &= h_{i}^{-1} \bigl[ \alpha (\Ss_{m}') \cap \alpha(\mcl{B}_{2} 
	    \setminus \mcl{Y}_{m}) )  \bigr] \\
	  &= h_{i}^{-1} \left[ \alpha (\Ss_{m}') \cap \Bigl( C_{2} 
	    \setminus \alpha 
	      \bigl( Exp \bigl( \mbf{F}_{[0,R_{m}]}[\overline{\clU}] \bigr) \Bigr) \notag \right] \\
	  &= h_{i}^{-1} \left[ \alpha (\Ss_{m}') \cap \Bigl( C_{2} 
	    \setminus \mbf{F}_{[0,R_{m}]}[\overline{\clU}] \Bigr) \right] \notag.
  \end{align}
By \eqref{E:C.sub.s.in.Bold.F.[0,s)}, $C_{2} \subset \mbf{F}_{[0,2)}[\clU]$. By \eqref{E:Rm.=.dist.Sm.P}, $R_{m} \in (0,1)$. Therefore,
    \begin{align*}
      C_{2} \setminus \mbf{F}_{[0,R_{m}]}[\overline{\clU}] 
        &= \bigl( C_{2} \cap \mbf{F}_{[0,2)}[\clU] \bigr) 
          \cap \mbf{F}_{[0,R_{m}]}[\overline{\clU}]^{c} \\
        &= C_{2} 
          \cap \bigl( \mbf{F}_{[0,2)}[\clU] \cap \mbf{F}_{[0,R_{m}]} [\overline{\clU}]^{c} \bigr) \\
        &= C_{2} \cap \mbf{F}_{(R_{m}, 2)}[\clU] .
   \end{align*}
Combining this with \eqref{E:hi.invrs.S2.eps.re-expression} and the as yet unproven \eqref{E:hi.invrs.theta.m}, we get
  \begin{align}  \label{E:hi.invrs.S2.eps.inclusion}
	h_{i}^{-1} (\mathscr{S}_{2,m}) 
	  &= h_{i}^{-1} \left[ \alpha (\Ss_{m}') 
	    \cap \Bigl( C_{2} \cap \mbf{F}_{(R_{m}, 2)}[\clU] \Bigr) \right] \notag \\
	  &\subset h_{i}^{-1} ( \mbf{F}_{(R_{m}, 2)}[\clU] ) \\
	    &\subset \mcl{A}_{i} \times \Bigl[ (\theta_{m}, 1) 
	      \cdot \bigl( \{1\} \times \msf{L}_{i} \bigr) \Bigr] \notag .
  \end{align}
(See \eqref{E:S.2m.subset.C2.less.F}.)
 
We may apply $h_{i}$ to both sides of \eqref{E:hi.invrs.theta.m} and get
	\begin{equation}  \label{E:hi.theta.m}
	   \mbf{F}_{(R_{m},2)}[\overline{\clU}] 
	     \subset h_{i} \left( \mcl{A}_{i} \times \Bigl[ (\theta_{m}, 1) \cdot \bigl( \{1\} 
	       \times \msf{L}_{i} \bigr) \Bigr] \right).                  
	\end{equation}

To prove \eqref{E:hi.invrs.theta.m}, let $\mcl{A}_{i}$ be one of the finite collection of $\mcl{A}$'s covering $\overline{\clU}$. (Since $\overline{\clU}$ is compact, it is covered by finitely many $\mcl{V}$'s as in definition \ref{D:fibering.by.cones}, part \ref{I:local.triv}. 
Let $(x, v)  \in h_{i}(\mcl{A}_{i} \times \msf{CL}_{i})$. Assume $|v| = 1$ (permissible by \eqref{E:eps.=.2}). 
Thus, there exists $s \in [0,1)$ and $z \in \msf{L}_{i}$ s.t.\ $(x,v) = h_{i}(x, s, sz)$. Recall \eqref{E:xi+.from.2.metrics}, \eqref{E:metric.on.Pf.x.CL}, and \eqref{E:lambda.metric.defn}. By  parts \ref{I:homogeneity.of.hi} and \ref{I:h.hi.invrs.Lip} of definition \ref{D:fibering.by.cones}, and \eqref{E:msfZ.unif.bounded}, for each $i$ there exists 
$M_{i} \in (1, \infty)$ s.t.\ 
	\begin{align}   \label{E:M'.s.and.v}
		1 &= |v| = \xi_{+} \bigl[ (x,0), (x,v) \bigr] \notag \\ 
		   &= \xi_{+} \bigl[ h_{i}(x, 0, 0 \cdot z), h_{i}(x, s, sz) \bigr] \notag \\
		   &\leq M_{i} (\xi \times \lambda_{i}) \bigl[ (x, 0, 0 \cdot z), (x, s, sz) \bigr]  \\            
		   &= M_{i} \, \lambda_{i} \bigl( (0, 0 \cdot z), (s, sz) \bigr) \notag  \\           
		   &= M_{i} \, \bigl| (0, 0) - (s, sz) \bigr| \notag \\
		   &\leq M_{i} \bigl( s + s |z| \bigr) \leq 3 M_{i} \, s. \notag            
	\end{align}
Thus, $s \geq 1/(3 M_{i})$. 

Let 
    \begin{equation*}
      \tilde{\theta} \in (0, 1) \text{ be the minimum value of } 1/(3 M_{i}) 
    \end{equation*} 
over the finite number of $\mcl{A}_{i}$'s included in \eqref{E:clU.bar.convered.by.finitely.many.Ai's}. Then for all those $\mcl{A}_{i}$'s we have 
$s \geq \tilde{\theta}$ 
if $(x, v)  \in h_{i} \Bigl( \mcl{A}_{i} \times s \cdot \bigl( \{ 1 \} \times \msf{L}_{i} \bigr) \Bigr)$  
with $|v| = 1$. Let 
    \begin{equation}  \label{E:theta.m.defn}
      \theta_{m} = R_{m} \tilde{\theta}, \quad (m = 1, 2, \ldots) . 
    \end{equation}
Since 
$\tilde{\theta} > 0$, trivially 
    \begin{equation}  \label{E:theta.m.asymp.R.m}
      \theta_{m} \asymp R_{m} .
    \end{equation}. 
By \eqref{E:Rm.=.dist.Sm.P}, 
$\theta_{m} \in (0,1)$ as required by \eqref{E:theta.m.in.(0,1)}.
Recall \eqref{E:boldF.[E].defns}. 
Suppose $(x,v) \in \mbf{F}_{(R_{m},2)}[\overline{\clU} \cap \mcl{A}_{i}]$, so $|v| > R_{m}$. 
Write $(x,v) = h_{i}(x, s, sz)$. Then, by \eqref{E:extended.homogeneity.of.hi}, 
$1 = |v|^{-1} \bigl| (x,v) \bigr| = \Bigl| h_{i} \bigl( x, s |v|^{-1}, s |v| ^{-1} z \bigr) \Bigr|$. Hence, 
	\begin{equation}  \label{E:s.|v|.invrs.lwr.bound}
		s |v|^{-1} \geq \tilde{\theta}. 
	\end{equation}
Therefore, $s \geq |v| \tilde{\theta} \geq R_{m} \tilde{\theta} = \theta_{m}$. This proves the claim \eqref{E:theta.m.in.(0,1)} -- \eqref{E:hi.invrs.theta.m}.

We introduced $\threeess$ in \eqref{E:AB.dilate.defn} as an unspecified constant 
$\geq 1$. Now we finally specify it. Release $t$ from its definition as $\dim \T$. We \emph{claim:} There exists a constant $\threeess \in [1, \infty)$ s.t.\ 
	\begin{multline} \label{E:threeess.appears}
		\text{If } x \in \mcl{A}_{i}, \, t \in \bigl( \rho(x) R \tilde{\theta}, 1 \bigr), 
		     \, z \in \msf{L}_{i}, \text{  and } (x,v) := h_{i}(x, t, tz). \\
		    \text{ Then } |v| \geq R \rho(x)/\threeess, 
		      \text{ for all } i \text{ and } R \in (0,1).  
	\end{multline}
To prove \eqref{E:threeess.appears}, let $x \in \mcl{A}_{i}$, $t \in [0,1)$, 
$z \in \msf{L}_{i}$, and let $(x,v) := h_{i}(x, t, tz)$. Then, by \eqref{E:metric.on.Pf.x.CL}, \eqref{E:lambda.metric.defn}, and \eqref{E:xi+.xi.and.|.|}, since $h_{i}^{-1}$ is Lipschitz (by part \ref{I:h.hi.invrs.Lip} of definition \ref{D:fibering.by.cones}), there is a constant 
$N_{i} \in (0, \infty)$ s.t.\ 
	\begin{equation*}   \label{E:N.s.and.v}
		t \leq t \bigl| (1,z) \bigr| 
		  = (\xi \times \lambda_{i}) \bigl[ (x, t, tz), (x, 0, 0) \bigr] 
		  \leq N_{i} \, \xi_{+} \bigl[ (x, v), (x, 0) \bigr] = N_{i} |v|.
	\end{equation*}
Let $\bar{N}$ denote the largest $N_{i}$ corresponding to all the finitely many 
$\mcl{A}_{i}$'s covering all the $\mcl{V}$'s intersecting 
$\overline{\clU}$. So $\bar{N} < \infty$ and $|v| \geq t/\bar{N}$. 
Let 
    \begin{equation}  \label{E:threeess.defn}
      \threeess := \max \{1, \bar{N}/\tilde{\theta} \} .
    \end{equation}
Then, if $t > \rho(x) R \tilde{\theta}$ we have that 
$|v| \geq R \, \rho(x)/\threeess$. This proves the claim \eqref{E:threeess.appears}. Since $R_{m} \to 0$, by \eqref{E:Rm.=.dist.Sm.P}, we may assume 
$0 < R_{m} < \min(1, \threeess/3)$. Therefore, the conclusions of \eqref{E:R.<.threeess/3.A.B.ineqs} and lemma \ref{L:f.dil.Lipschitz} hold with $R = R_{m}$.

Let $\msf{L} := \msf{L}_{i}$. Then $\msf{L}$ is a stratified space so it is the disjoint union of finitely many strata $\msf{Z}$. By definition \ref{D:fibering.by.cones}, part \ref{I:L.A.tameness}, $X := \msf{L}$ satisfies \eqref{E:tameness.of.stratified.space}. Therefore, since $\msf{L}$ is compact, $\msf{Z}$ is covered by finitely many tractable coordinate neighborhoods, 
$\mbf{Z}$. If $t \in [0,1)$, let 
    \begin{equation}  \label{E:bold.Z.tilde.defn}
       \tilde{\mbf{Z}}^{t} := (t, 1) \cdot \bigl( \{1\} \times \mbf{Z} \bigr) 
          := \bigl\{ (s, sz) \in \msf{CL} : s \in (t, 1), z \in \mbf{Z} \bigr\}. 
            \text{ Let } \tilde{\mbf{Z}} := \tilde{\mbf{Z}}^{\theta_{m}} .
    \end{equation} 
Note that if $\mbf{Z}_{1}$ and $\mbf{Z}_{2}$ are coordinate neighborhoods of (possibly the same) strata of $\msf{L}$ then 
    \begin{equation*} \label{E:if.tilde.Z.thetas.overlap.so.do.Zs}
        \text{If } \tilde{\mbf{Z}}_{1}^{\theta_{m}} \cap \tilde{\mbf{Z}}_{2}^{\theta_{m}} 
          \neq \varnothing 
            \text{ then } \mbf{Z}_{1} \cap \mbf{Z}_{2} \neq \varnothing .
    \end{equation*}
If $y \in \Pf$, $t > 0$, and $z \in \mbf{Z}$ define 
$\tau \bigl( y, t(1, z) \bigr) := (y, z) \in \Pf \times \mbf{Z}$ so 
$\tau \bigl( y, t(1, z) \bigr)$ is constant in $t > 0$ and $\tau$ does not depend on $m$. 
Recall \eqref{E:indicator.fn.defn}. Hence, if $\mbf{A} \subset \Pf$, 
    \begin{equation}  \label{E:tau.constant.in.m}
            1^{ \mbf{A} \times \tilde{\mbf{Z}}^{\theta_{m}} } 
              \leq (1^{ \mbf{A} \times \mbf{Z} }) \circ \tau .
    \end{equation}

Let $q := \dim \msf{Z}$. Since $\msf{Z}$ is a stratum of $\msf{L}$, by part \ref{I:L.d-p-1} of definition \ref{D:fibering.by.cones}, 
	\begin{equation} \label{E:q.leq.d-p-1}
		q \leq d-p-1. 
	\end{equation}
	
Let $\mbf{Z}$ be a tractable coordinate neighborhood of a stratum, $\msf{Z}$, of $\msf{L}$. Then
    \begin{multline}  \label{E:psi.paramz.Z}
         \text{There exists } \psi :  D \to \mbf{Z} \subset \RR^{J} \text{ s.t.\ } D \subset \RR^{q}
           \text{ is open, bounded, and convex} \\
             \text{ and $\psi$ is a smooth bi-Lipschitz bijection.}
    \end{multline} 
(This holds trivially if $q = 0$. Do not confuse $D$ and $\D$.) Since $D$ is bounded we may replace $\psi^{-1}$ by $\beta \circ \psi^{-1}$, where $\beta : \RR^{q} \to \RR^{q}$ is an appropriate affine function to arrange things so that 
	\begin{equation}  \label{E:1/3.<.b.lngth.<.2/3}
		  1/3 < |b| <  2/3 \text{ for every } b \in D.
	\end{equation}
To see this, pick $c > 0$ s.t.\ $|c b| < 1/6$ for every $b \in D$. 
Pick $v \in \RR^{q}$  s.t.\ $|v| = 1/2$. Let $b \in D$. Then
    \begin{equation*}
      1/3 = 1/2 - 1/6 < |v| - |cb| \leq |cb+v| \leq |v| + |cb| < 1/2 + 1/6 = 2/3 .
    \end{equation*}
Thus, the affine transformation $b \mapsto cb+v$ gives the desired result.
($\mbf{Z}$ is a tractable subset of a stratum $\msf{Z}$ of the ``link'' $\msf{L}$.) By \eqref{E:msfZ.unif.bounded}, we may assume $|z| < 2$ for every $z \in \mbf{Z}$. Then $\tilde{\mbf{Z}}$ is an $(q+1)$-dimensional manifold and is parametrized by 
$(s,b) \mapsto s \bigl( 1, \psi(b) \bigr) = s \cdot \bigl( 1, \psi(b) \bigr) 
= \bigl( s, s \psi(b) \bigr)$, 
where $s \in (\theta_{m},1)$, $b \in D$

In \eqref{E:clU.bar.convered.by.finitely.many.Ai's} we noted that only finitely many 
$\mcl{A}_{i}$'s are needed to cover $\overline{\clU}$. Let $\mcl{A}$ be one of them. From subsection \ref{SSS:conical.fibers}, we know that there are only finitely many strata $\Rcl$ 
of $\Pf$. Let $\Rcl$ be one. It is a smooth manifold of dimension no greater than $p$.  By part \ref{I:local.triv} of definition \ref{D:fibering.by.cones}, 
$\mcl{A} \cap \Rcl$ is open in $\Rcl$, Let $\ell := \dim (\mcl{A} \cap \Rcl) = \dim \Rcl \leq p$. By definition \ref{D:fibering.by.cones}, part \ref{I:L.A.tameness}, $X := \mcl{A}$ satisfies \eqref{E:tameness.of.stratified.space}. Therefore, there are finitely many tractable coordinate neighborhoods of $\mcl{A} \cap \Rcl$ that cover 
$\mcl{A} \cap \Rcl \cap \overline{\clU}$. We conclude that 
    \begin{equation*}
      \text{A finite collection of tractable neighborhoods of $\mcl{A}_{i} \cap \Rcl$'s suffices to cover } 
        \overline{\clU} .
    \end{equation*}
 
If $\mbf{A}$ is one of these tractable neighborhoods then,
    \begin{multline}  \label{E:eta.paramz.A}
         \text{There exists } \eta : E \to \mbf{A} \text{ s.t.\ } E \subset \RR^{\ell} 
           \text{ is open, bounded, and convex} \\
             \text{ and $\eta$ is a smooth bi-Lipschitz bijection.}
    \end{multline} 
It is w.r.t.\ to the Euclidean metric on $\RR^{\ell}$ and $\xi$ (see \eqref{E:xi.is.metric.on.D}) 
that $\eta$ is bi-Lipschitz. 

Thus, there are only finitely many sets 
$\mbf{A}$ and, hence, only finitely many maps $\eta : E \to \mbf{A}$ that we have to worry about. That implies that there are only finitely many sets $\mbf{Z}$ and, hence, only finitely many maps $\psi : D \to \mbf{Z}$ that we have to worry about. It follows that there are Lipschitz constants for $\psi$ and $\eta$ that only depend only on $\D$, $C[\Pf]$, $\T$, and $a$, not specifically on $\mbf{A}$ and $\mbf{Z}$.  

By \eqref{E:bold.Z.tilde.defn}, 
$\tilde{\mbf{Z}}^{0} := \bigl\{ (s, sz) \in \msf{CL} : s \in (0, 1), z \in \mbf{Z} \bigr\}$. 
A consequence of all this is the following.
    \begin{equation} \label{E:finitely.many.A.x.Ztilde}
        C[\overline{\clU}] \setminus \Pf \text{ is covered by finitely many sets of the form } 
          h( \mbf{A} \times \tilde{\mbf{Z}}^{0} ) ,
    \end{equation}
where $\mbf{A}$ is a tractable coordinate neighborhood of some 
$\mcl{A}_{i} \cap \Rcl \subset \Pf$ for some stratum of $\Pf$, $h = h_{i}$, and $\mbf{Z}$ is a tractable coordinate neighborhood of a stratum of the corresponding $\msf{L}_{i}$. 

Re-index the finite collection of $\mcl{A}_{i}$'s that cover $\overline{\clU}$ and corresponding $h$'s and $\msf{L}$'s. (One or more $\mcl{A}_{i}$'s may appear more than once on the list, but with distinct $h_{i}$'s.) Recall the definition of indicator function, \eqref{E:indicator.fn.defn}. Define a function $\mathfrak{g}$ on $C[\overline{\clU}]$ by
    \begin{equation*}
        \mathfrak{g} := \mathfrak{g}_{m} := \sum_{\Rcl} \sum_{i} 
          \sum_{\mbf{A} \subset \mcl{A}_{i} \cap \Rcl, \, \mbf{A} \cap \overline{\clU} 
            \neq \varnothing} 
              \; \sum_{\msf{Z} \subset \msf{L}_{i}}
                \sum_{\mbf{Z} 
                  \subset \msf{Z}} 1^{h_{i}(\mbf{A} \times \tilde{\mbf{Z}}^{\theta_{m}})},
    \end{equation*}
where the outer sum is over the finite collection of strata, $\Rcl$, of $\Pf$ that intersect 
$\overline{\clU}$. Thus, the sum in the preceding is finite and takes values in the set of non-negative integers. Moreover, by \eqref{E:C.sub.s.in.Bold.F.[0,s)} and \eqref{E:hi.theta.m}, 
    \begin{equation}  \label{E:frak.g.geq.1}
      \mathfrak{g}(x,v) \geq 1 \text{ for every }
        (x,v) \in C_{\twoess} \setminus \mbf{F}_{[0,R_{m}]}[\overline{\clU}] .
    \end{equation} 
By definition \ref{D:fibering.by.cones} part \ref{I:L.d-p-1}, the $h_{i}$'s are injective. Therefore, 
    \begin{equation}  \label{E:1.h=1.circ.h.invrs}
      1^{h_{i}(\mbf{A} \times \tilde{\mbf{Z}}^{\theta_{m}})} 
        = (1^{ \mbf{A} \times \tilde{\mbf{Z}}^{\theta_{m}} }) \circ h_{i}^{-1} .
    \end{equation}
The sets $\mbf{A} \times \tilde{\mbf{Z}}$ are Borel measurable. (Each is the product of inverse images, under coordinate maps, of open sets). It follows that the functions 
$1^{h_{i}(\mbf{A} \times \tilde{\mbf{Z}}^{\theta_{m}})}$ are Borel measurable. And hence so 
is $\mathfrak{g}$. Moreover, by \eqref{E:1.h=1.circ.h.invrs} and \eqref{E:tau.constant.in.m},
    \begin{equation*}
        \mathfrak{g} \leq \sum_{\Rcl} \sum_{i} 
          \sum_{\mbf{A} \subset \mcl{A}_{i} \cap \Rcl, \, \mbf{A} \cap \overline{\clU} 
            \neq \varnothing} 
              \; \sum_{\msf{Z} \subset \msf{L}_{i}}
                \sum_{\mbf{Z} \subset \msf{Z}} (1^{ \mbf{A} \times \mbf{Z} }) 
                  \circ \tau \circ h_{i }^{-1}. 
    \end{equation*} 

Let $\kappa$ be the maximum of the right hand side (RHS) of the preceding. 
Then $\kappa < \infty$ 
and $\kappa$ \emph{does not depend on} $\{ \Phi_{m}, \Ss_{m}', R_{m} \}$. Let 
$\mathscr{B} \subset C_{\twoess} \setminus \mbf{F}_{[0,R_{m}]}$ be Borel measurable. Then, by \eqref{E:Cs.defn} and \eqref{E:rho.on.W.off.U}, we have 
$\pi(\mathscr{B}) \subset \clU$. Therefore, by \eqref{E:frak.g.geq.1}, 
    \begin{equation}  \label{E:Hma.S2.bounds}
        \kappa \, \Hm^{a} (\mathscr{B}) \geq \int_{\mathscr{B}} \mathfrak{g} \; d\Hm^{a} 
                    \geq \Hm^{a} (\mathscr{B}) .
    \end{equation}
But
    \begin{equation}  \label{E:int.g.=.sum}
        \int_{\mathscr{B}} \mathfrak{g} \; d\Hm^{a} 
          = \sum_{\Rcl} \sum_{i} 
            \sum_{\mbf{A} \subset \mcl{A}_{i} \cap \Rcl, \, \mbf{A} \cap \overline{\clU} 
              \neq \varnothing}  
                \; \sum_{\msf{Z} \subset \msf{L}_{i}}
                  \sum_{\mbf{Z} \subset \msf{Z}} 
                    \Hm^{a} \bigl[ h_{i}(\mbf{A} \times \tilde{\mbf{Z}}^{\theta_{m}})
                     \cap \mathscr{B} \bigr] .
    \end{equation}
    
Recall the definition, \eqref{E:F.on.C[Pf].defn}, of $F_{R}$. Define
    \begin{equation}  \label{E:F:=.alpha.f.alpha.invrs}  
          F := F_{m} := F^{R_{m}} .
    \end{equation}
(See \eqref{E:F.from.f.dil}.) By definition \ref{D:fibering.by.cones} part \ref{I:Exp.alpha.homeom}, $\alpha$ and its inverse are Borel measurable functions. Therefore, by lemma \ref{L:f.dil.Lipschitz}, $F$ and $F^{-1}$ are Borel. Hence, by 
\eqref{E:S2.Borel.finite.meas}, $\mathscr{B}_{1} := \mathscr{S}_{2,m}$ and 
$\mathscr{B}_{2} := F( \mathscr{S}_{2,m} ) = (F^{-1})^{-1}(\mathscr{S}_{2,m})$ 
are two possible $\mathscr{B}$'s.

\emph{Suppose} for some $K < \infty$ independent of $m$, etc., we have for every pair of tractable coordinate neighborhoods 
$\mbf{A}$ and $\mbf{Z}$, 
    \begin{equation}  \label{E:ineq.restricted.to.AxZ}
            \Hm^{a} \bigl[ h_{i}(\mbf{A} \times \tilde{\mbf{Z}}^{\theta_{m}}) 
              \cap F( \mathscr{S}_{2,m} ) \bigr]  
               \leq K R_{m}^{p-d+1} 
                 \,\Hm^{a} \bigl[ h_{i}(\mbf{A} \times \tilde{\mbf{Z}}^{\theta_{m}}) 
                   \cap \mathscr{S}_{2,m} \bigr] .
    \end{equation}
Then, by \eqref{E:Hma.S2.bounds} and \eqref{E:int.g.=.sum},
    \begin{align*}
      \Hm^{a} \bigl( F( \mathscr{S}_{2,m} ) \bigr) 
        & \leq \sum_{\Rcl} \sum_{i} 
          \sum_{\mbf{A} \subset \mcl{A}_{i} \cap \Rcl, \, \mbf{A} \cap \overline{\clU} 
            \neq \varnothing}  
              \; \sum_{\msf{Z} \subset \msf{L}_{i}}
                \sum_{\mbf{Z} \subset \msf{Z}} 
                  \Hm^{a} \bigl[ h_{i}(\mbf{A} \times \tilde{\mbf{Z}}^{\theta_{m}}) 
                    \cap F ( \mathscr{S}_{2,m} ) \bigr] \\
          &\leq  K R_{m}^{p-d+1} \, 
            \sum_{\Rcl} \sum_{i} \sum_{\mbf{A} \subset \mcl{A}_{i} \cap \Rcl, \, \mbf{A} 
              \cap \overline{\clU} \neq \varnothing}  
                \; \sum_{\msf{Z} \subset \msf{L}_{i}}
                  \sum_{\mbf{Z} \subset \msf{Z}} 
                    \Hm^{a} \bigl[ h_{i}(\mbf{A} \times \tilde{\mbf{Z}}^{\theta_{m}}) 
                      \cap \mathscr{S}_{2,m} \bigr] \\
          &\leq \kappa K R_{m}^{p-d+1} \, \Hm^{a}(\mathscr{S}_{2,m}) .
    \end{align*} 
Therefore, in order to prove \eqref{E:dilate.alpha.ineq}, \emph{it suffices to show} \eqref{E:ineq.restricted.to.AxZ} for every tractable $\mbf{A}$ and $\mbf{Z}$.

\emph{Until further notice let $\msf{Z}$ be a fixed stratum of a link $\msf{L}_{i}$ and and let 
$\mbf{Z}$ be a fixed tractable neighborhood in $\msf{Z}$. Except where noted, 
let $\Rcl$, $i$, and $\mbf{A}$, a tractable neighborhood of $\Rcl \cap \mcl{A}_{i}$, also be fixed.}

\emph{Second application of} \eqref{E:general.change.of.variables}: Now we use \eqref{E:general.change.of.variables} to change the problem about events in $C[\clU]$ to one concerning events in nice subsets of subsets of the form 
$ \mcl{A}_{i} \times \msf{CL}_{i}$. (See definition \ref{D:fibering.by.cones}.)
If $\bigl( y, (s, sz) \bigr) \in \{y\} \times \msf{CL}_{i}$ then, by definition \ref{D:fibering.by.cones} part \ref{I:L.d-p-1}, $h_{i} \bigl( y, (s, sz) \bigr) \in C[y]$. Hence, by \eqref{E:F(C[P])=C[P]}, 
$F_{m} \circ h_{i}\bigl( y, (s, sz) \bigr) \in C[y]$. Therefore, 
$h_{i}^{-1} \circ F_{m} \circ h_{i}\bigl( y, (s, sz) \bigr) \in \{y\} \times \msf{CL}_{i}$.
Thus, for given $i$ and $m$, 
	\begin{equation}   \label{E:F.i.Rm.defn}
	  \mathscr{F}_{R_{m}} := \mathscr{F}_{i, R_{m}} := h_{i}^{-1} \circ F_{m} \circ h_{i}
	    : \mcl{A}_{i} \times \msf{CL}_{i} \to \mcl{A}_{i} \times \msf{CL}_{i}. 
	\end{equation}
$\mathscr{F}_{i, R_{m}}$ is defined everywhere 
on $\mcl{A}_{i} \times \msf{CL}_{i}$. By parts \ref{I:L.d-p-1}, \ref{I:h.hi.invrs.Lip}, and \ref{I:Exp.alpha.homeom} of definition \ref{D:fibering.by.cones};  
lemma \ref{L:f.dil.Lipschitz}; \eqref{E:comp.of.Lips.is.Lip}, and \eqref{E:F.is.injective}; we have that $\mathscr{F}_{i, R_{m}}$ is injective and bi-Lipschitz. 

Recall \eqref{E:S.2.defn}. \emph{Suppose} for some constant $K_{5} < \infty$ the following holds for every $i$, $m$, and every tractable $\mbf{A}$ and $\mbf{Z}$ belonging to $\mcl{A}_{i}$, $\msf{L}_{i}$ appropriately,
   \begin{equation}   \label{E:dilate.ineq.in.C[U].with.h.invrs}
       \Hm^{a} \left[ \mathscr{F}_{i, R_{m}} 
	\Bigl( h_{i}^{-1} \bigl[ h_{i}(\mbf{A} \times \tilde{\mbf{Z}}^{\theta_{m}}) 
	  \cap \mathscr{S}_{2,m} \bigr] \Bigr) \right] 
	\leq K_{5} R_{m}^{p-d+1}
	  \Hm^{a} \Bigl( h_{i}^{-1} \bigl[ h_{i}(\mbf{A} \times \tilde{\mbf{Z}}^{\theta_{m}}) 
	    \cap \mathscr{S}_{2,m} \bigr] \Bigr).
   \end{equation}
(Compute $\Hm^{a}$ based on the metric $\xi \times \lambda$ 
defined in \eqref{E:metric.on.Pf.x.CL}.)

Since $h_{i}$ is bi-Lipschitz by part \ref{I:h.hi.invrs.Lip} of definition \ref{D:fibering.by.cones}, by \eqref{E:general.change.of.variables} 
with $S = h_{i}(\mbf{A} \times \tilde{\mbf{Z}}^{\theta_{m}}) \cap \mathscr{S}_{2,m}$. 
$f = F$, and $\phi := h_{i}^{-1}$, by \eqref{E:dilate.ineq.in.C[U].with.h.invrs}, the following must hold for some constant $K < \infty$ and all $i$, etc.,
    \begin{equation} \label{E:not.quite.rite.alpha.f.alpha.invrs.ineq}
      \Hm^{a} \Bigl( F 
        \bigl[ h_{i}(\mbf{A} \times \tilde{\mbf{Z}}^{\theta_{m}}) 
          \cap \mathscr{S}_{2,m} \bigr] \Bigr)
            \leq K R_{m}^{p-d+1} \Hm^{a} \bigl[ h_{i}(\mbf{A} \times \tilde{\mbf{Z}}^{\theta_{m}}) 
              \cap \mathscr{S}_{2,m} \bigr] .
    \end{equation}
(Compute $\Hm^{a}$ based on $\xi_{+}$ defined in \eqref{E:xi+.from.2.metrics}.)

By \eqref{E:g.commutes.w/.set.ops}, we have
    \begin{equation}  \label{E:alpha.f.alpha.invrs.distributes}
      F \bigl[ h_{i}(\mbf{A} \times \tilde{\mbf{Z}}^{\theta_{m}}) \cap \mathscr{S}_{2,m} \bigr] 
         = F \bigl[ h_{i}(\mbf{A} \times \tilde{\mbf{Z}}^{\theta_{m}}) \bigr]
           \cap F(\mathscr{S}_{2,m}) .
    \end{equation}
    \emph{Claim:} 
        \begin{equation} \label{E:alpha.f.alpha.invrs.maps.h(AxZ.theta).into.self}
           F \bigl[ h_{i}(\mbf{A} \times \tilde{\mbf{Z}}^{\theta_{m}}) \bigr]
             \subset h_{i}(\mbf{A} \times \tilde{\mbf{Z}}^{\theta_{m}}) .
        \end{equation}
 To see this, let $\bigl( y, s(1,z) \bigr) \in \mbf{A} \times \tilde{\mbf{Z}}^{\theta_{m}}$ 
 and let $(y,v) = h_{i} \Bigl[ \bigl( y, s(1,z) \bigr) \Bigr] \in C[\mbf{A}]$. By \eqref{E:F(C[P])=C[P]}, $F(y,v) \in C[\mbf{A}]$ and, by \eqref{E:f.dil.is.dilation},
    \begin{equation}  \label{E:there.exists.r}
      \text{there exists } r = r(y,v) \geq 1 \text{ s.t.\ } F(y,v) = (y, rv) \in C[\mbf{A}] .
    \end{equation} 
 Now, by  \eqref{E:homogeneity.of.hi}, 
 $(y, v) = h_{i} \Bigl[ \bigl( y, r^{-1} r s(1,z) \bigr) \Bigr] = r^{-1} 
 h_{i} \Bigl[ \bigl( y, r s(1,z) \bigr) \Bigr]$. Hence, by \eqref{E:extended.homogeneity.of.hi}, 
        \begin{equation*}
          F(y,v) = (y, rv) = h_{i} \Bigl[ \bigl( y, r s(1,z) \bigr) \Bigr] .
        \end{equation*}
Now, $s > \theta_{m}$ so, by \eqref{E:there.exists.r}, $r s > \theta_{m}$. 
But, $(y, rv) \in C[\mbf{A}]$. Therefore, by \eqref{E:h(y,(1,z)).notin.C[P]}, 
we must have $r s < 1$. Hence, by \eqref{E:bold.Z.tilde.defn}, 
$\bigl( y, r s(1,z) \bigr) \in \mbf{A} \times \tilde{\mbf{Z}}^{\theta_{m}}$. This completes the proof of the claim \eqref{E:alpha.f.alpha.invrs.maps.h(AxZ.theta).into.self}.
    
We next \emph{claim:}
     \begin{equation} \label{E:hi.cap.F(S2).ineq}
           h_{i}(\mbf{A} \times \tilde{\mbf{Z}}^{\theta_{m}}) 
                  \cap F( \mathscr{S}_{2,m} ) 
            \subset F \bigl[ h_{i}(\mbf{A} \times \tilde{\mbf{Z}}^{\theta_{m}})  \bigr]
                  \cap F( \mathscr{S}_{2,m} ) .
    \end{equation}
Once we have prove this we can conclude, from it and \eqref{E:alpha.f.alpha.invrs.maps.h(AxZ.theta).into.self}, that 
        \begin{equation} \label{E:hi.cap.F(S2)=F(hi).cap.F(S2)}
           h_{i}(\mbf{A} \times \tilde{\mbf{Z}}^{\theta_{m}}) 
                  \cap F( \mathscr{S}_{2,m} ) 
              = F \bigl[ h_{i}(\mbf{A} \times \tilde{\mbf{Z}}^{\theta_{m}})  \bigr]
                  \cap F( \mathscr{S}_{2,m} ) .
        \end{equation}

The proof of \eqref{E:hi.cap.F(S2).ineq} is a bit delicate. Let 
$(y,v) \in h_{i}(\mbf{A} \times \tilde{\mbf{Z}}^{\theta_{m}}) \cap F( \mathscr{S}_{2,m} )$.
Since $(y,v) \in h_{i}(\mbf{A} \times \tilde{\mbf{Z}}^{\theta_{m}})$, we have $y \in \mbf{A}$ (by definition \ref{D:fibering.by.cones} part \ref{I:pi.h.(y,w)=y}) and, by \eqref{E:bold.Z.tilde.defn}, for some $z \in \mbf{Z}$ and $s \in (\theta_{m} ,1)$ we have 
    \begin{equation*}
      (y,v) = h_{i}(y, s, sz) .
    \end{equation*}
By \eqref{E:F:=.alpha.f.alpha.invrs} and \eqref{E:f.dil.is.dilation}, there exists 
$t \in (0,1]$ s.t.\ $F(y, tv) = (y,v)$. By \eqref{E:homogeneity.of.hi}, 
    \begin{equation*}
      (y, tv) = h_{i}(y, ts, tsz) .
    \end{equation*}

We also have $(y,v) \in F( \mathscr{S}_{2,m} )$. Recall that, by definition \ref{D:fibering.by.cones} part \ref{I:Exp.alpha.homeom}, 
$\alpha = Exp^{-1}$ on $\mcl{C}$. Thus, from \eqref{E:S.2.defn}, 
we have $\mathscr{S}_{2,m} \subset C_{2} \setminus \mbf{F}_{[0,R_{m}]}[\overline{\clU}]$.
Applying \eqref{E:C.sub.s.in.Bold.F.[0,s)} and \eqref{E:boldF.[E].defns} we get, 
    \begin{equation} \label{E:S.2m.subset.C2.less.F}
       \mathscr{S}_{2,m} \subset C_{2} \setminus \mbf{F}_{[0,R_{m}]}[\overline{\clU}] 
         \subset \mbf{F}_{(R_{m}, 2)}[\clU]  .
    \end{equation}
Thus, 
$(y,v) \in F \bigl( \mbf{F}_{(R_{m}, 2)}[\clU] \bigr)$. I.e., 
    \begin{equation*}
      F \bigl[ h_{i}(y, ts, tsz) \bigr] = F(y, tv) = (y,v) 
        \in F \bigl( \mbf{F}_{(R_{m}, 2)}[\clU] \bigr) .
    \end{equation*}
But $F$ is injective, so $h_{i}(y, ts, tsz) = (y, tv) \in \mbf{F}_{(R_{m}, 2)}[\clU] $. 
I.e., $(y, ts, tsz) \in h_{i}^{-1} \bigl( \mbf{F}_{(R_{m}, 2)}[\clU] \bigr)$. 
Therefore, by \eqref{E:hi.invrs.theta.m}, $ts \in (\theta_{m}, 1)$. 
Since $y \in \mbf{A}$ and $z \in \mbf{Z}$, we thus have, by \eqref{E:bold.Z.tilde.defn} again,  
$(y, ts, tsz) \in \mbf{A} \times \tilde{\mbf{Z}}^{\theta_{m}}$. Hence,
    \begin{equation*}
      (y,v) = F(y, tv) = F \bigl[ h_{i}(y, ts, tsz) \bigr] 
         \in F(\mbf{A} \times \tilde{\mbf{Z}}^{\theta_{m}}) .
    \end{equation*}
We are given 
$(y,v) \in h_{i}(\mbf{A} \times \tilde{\mbf{Z}}^{\theta_{m}}) \cap F( \mathscr{S}_{2,m} )$. This proves the claim \eqref{E:hi.cap.F(S2).ineq}. 
 
Combining \eqref{E:hi.cap.F(S2)=F(hi).cap.F(S2)}, \eqref{E:alpha.f.alpha.invrs.distributes}, and \eqref{E:not.quite.rite.alpha.f.alpha.invrs.ineq} we get 
    \begin{multline*}  
         \Hm^{a} \bigl[ h_{i}(\mbf{A} \times \tilde{\mbf{Z}}^{\theta_{m}}) 
                  \cap F( \mathscr{S}_{2,m} ) \bigr] \\
            = \Hm^{a} \Bigl( F \bigl[ h_{i}(\mbf{A} \times \tilde{\mbf{Z}}^{\theta_{m}})  \bigr]
                  \cap F( \mathscr{S}_{2,m} ) \Bigr) \
            = \Hm^{a} \Bigl( F 
              \bigl[ h_{i}(\mbf{A} \times \tilde{\mbf{Z}}^{\theta_{m}}) 
                \cap \mathscr{S}_{2,m}) \bigr] \Bigr) \\
                \leq K R^{p-d+1} \Hm^{a} \bigl[ h_{i}(\mbf{A} 
                  \times \tilde{\mbf{Z}}^{\theta_{m}}) \cap \mathscr{S}_{2,m} \bigr] .
    \end{multline*}
I.e., \eqref{E:ineq.restricted.to.AxZ} holds. Now, \eqref{E:hi.cap.F(S2)=F(hi).cap.F(S2)}, \eqref{E:alpha.f.alpha.invrs.distributes} are true regardless of whether \eqref{E:ineq.restricted.to.AxZ} holds or not. Therefore, since \eqref{E:not.quite.rite.alpha.f.alpha.invrs.ineq} is a consequence of \eqref{E:dilate.ineq.in.C[U].with.h.invrs}, 
    \begin{equation} \label{E:suffices.to.prov."dilate.ineq..."}
      \text{It suffices to prove } \eqref{E:dilate.ineq.in.C[U].with.h.invrs}.
    \end{equation}
 
The following is immediate from \eqref{E:alpha.f.alpha.invrs.maps.h(AxZ.theta).into.self} and \eqref{E:F.i.Rm.defn}, 
    \begin{equation} \label{E:FR.maps.A.x.Z.into.self}
      \mathscr{F}_{R} ( \mbf{A} \times \tilde{\mbf{Z}}^{\theta_{m}} ) 
        \subset \mbf{A} \times \tilde{\mbf{Z}}^{\theta_{m}} .
    \end{equation}
So in \eqref{E:dilate.ineq.in.C[U].with.h.invrs} we are working 
in $\mbf{A} \times \tilde{\mbf{Z}}^{\theta_{m}}$. 

Now we translate the problem into one concerning convex subsets of Euclidean space. Recall (see \eqref{E:eta.paramz.A} and \eqref{E:psi.paramz.Z}) that 
$E \subset \RR^{\ell}$ and $D \subset \RR^{q}$ are open, bounded, and convex, 
$\eta : E \to \mbf{A}$ parametrizes $\mbf{A}$, and $\psi : D \to \mbf{Z}$ parametrizes 
$\mbf{Z}$. Let 
    \begin{equation*}
        N := \ell+q+1 .
    \end{equation*}
If $\mathscr{J} \subset (0, \infty]$ let 
	\begin{multline}  \label{E:mathscrX.defn}
		\mathscr{X}_{\mathscr{J}} 
		  := \bigl\{ \bigl(e, t(1,b) \bigr) \in \RR^{N} : e \in E, t \in 
		    \mathscr{J}, b \in D \bigr\} .  \\     
		    \text{ If } r \in [0,1),  \text{ let } \mathscr{X}_{r} := \mathscr{X}_{(r,1)}
		      \text{ and } \mathscr{X} := \mathscr{X}_{0} . 
	\end{multline}
Put on $\mathscr{X}_{\mathscr{J}}$ the topology it inherits from $\RR^{N}$. 
Thus, if $\mathscr{J}$ is open in $(0, \infty)$, then 
$\mathscr{X}_{\mathscr{J}}$ is open. 

\emph{Claim:} If $0 \leq a < b \leq \infty$ then 
    \begin{equation} \label{E:mathscrX.(a,b).convex}
      \mathscr{X}_{(a,b)} \text{  is convex. }
    \end{equation} 
So $\mathscr{X}$ is an open convex subset of $\RR^{N}$. For let $(e_{i}, s_{i}, s_{i} b_{i}) \in \mathscr{X}_{(a,b)}$ ($i = 1,2$), let $\lambda \in [0,1]$ ,and consider $\lambda \, (e_{1}, s_{1}, s_{1} b_{1}) + (1-\lambda) (e_{2}, s_{2}, s_{2} b_{2})$. By \eqref{E:eta.paramz.A}, $e := \lambda e_{1} + (1-\lambda) e_{2} \in E$. 
Let $s := \lambda s_{1} + (1-\lambda) s_{2} \in (a,b)$.
Let $\mu := \lambda s_{1}/s \in [0,1]$ and let $b := \mu b_{1} + (1-\mu) b_{2}$. By \eqref{E:psi.paramz.Z}, $b \in D$. We have
    \begin{multline*}
      s b = s \mu b_{1} + s (1-\mu) b_{2} 
        = \lambda s_{1} b_{1} + s (1 - \lambda s_{1}/s) b_{2} \\
          = \lambda s_{1} b_{1} + (s - \lambda s_{1}) b_{2} 
            = \lambda s_{1} b_{1} + (1-\lambda) s_{2} b_{2} .
    \end{multline*}
Thus, 
    \begin{equation*}
      \lambda \, (e_{1}, s_{1}, s_{1} b_{1}) + (1-\lambda) (e_{2}, s_{2}, s_{2} b_{2})
        = (e, s, sb) \in \mathscr{X}_{(a,b)} 
    \end{equation*}
as desired.

Recall that, by \eqref{E:bold.Z.tilde.defn}, 
$\tilde{\mbf{Z}}^{0} := (0, 1) \cdot \bigl( \{1\} \times \mbf{Z} \bigr) 
:= \bigl\{ (t, tz) \in \msf{CL}_{i} : t \in (0, 1), z \in \mbf{Z} \bigr\}$.
Let $H := H_{\mbf{A}, \mbf{Z}}: \mathscr{X} \to \mbf{A} \times \tilde{\mbf{Z}}^{0}$ be the map 
   \begin{equation}  \label{E:H.map.defn}
	H \bigl( e, t(1,b) \bigr) := H_{\mbf{A}, \mbf{Z}} \bigl( e, t(1,b) \bigr) := 
	  \Bigl( \eta(e), t \bigr( 1, \psi(b) \bigr) \Bigr) 
	    \in \mbf{A} \times \tilde{\mbf{Z}}^{0}, 
	      \quad \bigl(e, t(1,b) \bigr) \in \mathscr{X}.
   \end{equation}
\emph{Claim:} $H$ is bi-Lipschitz w.r.t.\ the Euclidean metric on $\RR^{N}$ 
and $\xi \times \lambda$. Moreover, there are Lipschitz constants for $H$ and $H^{-1}$ that do not depend on $R$. Clearly, $H$ is Lipschitz since $\psi$ and $\eta$ are. (See example \ref{Ex:ratnl.fns.loc.Lip}. $t$ and $\psi(b)$ are bounded on $\mathscr{X}$.) The Lipschitz constants for $\psi$ and $\eta$ only depend on the geometry of $C[\Pf]$ and as such are independent of $R$.

It remains to show that $H$ is injective and $H^{-1}$ is Lipschitz. Let $\pi_{j}$ be projection onto the $j^{th}$ factor of $\bigl\{ (t, tz) : t \in (0,1),  z \in \mbf{Z} \bigr\}$ ($j = 1,2$). 
(Thus, $\pi_{1}(t, tz) := t$, $\pi_{2}(t, tz) := tz$.) $H^{-1}$ is given by
	\begin{multline}  \label{E:H.invrs.formla}
		H^{-1} : (y, \tilde{z}) \mapsto \biggl( \eta^{-1}(y), \; \pi_{1}(\tilde{z}) 
		  \cdot \Bigl(1, \psi^{-1} \bigl[ \pi_{1}(\tilde{z})^{-1} \cdot 
		    \pi_{2}(\tilde{z}) \bigr] \Bigr) \biggr), \\
		      \quad y \in \mbf{A}, \; \tilde{z} = (t, tz) \in \tilde{\mbf{Z}},
	\end{multline}
where here ``$\cdot$'' denotes scalar multiplication. 

The  Lipschitz property of the $\eta^{-1}$ component is already assumed. The  Lipschitz property of the $\pi_{1}$ is already trivial. By example \ref{Ex:ratnl.fns.loc.Lip} again, it then suffices to prove the Lipschitz-osity 
of $(t, \tilde{z}) \mapsto t \, \psi^{-1}( t^{-1} \tilde{z} )$ on 
$(0,1) \times \bigl[ (0,1) \cdot \mbf{Z} \bigr]$. 
Let $t_{j} \in (0,1)$ and $\tilde{z}_{j} \in t_{j} \mbf{Z}$ ($j = 1,2$). 
Thus, $(t_{j}, \tilde{z}_{j}) \in \msf{CL}_{i}$. Then, by \eqref{E:msfZ.unif.bounded},  we have 
$|\tilde{z}_{2}| < 2t_{2}$. Let $L < \infty$ be a Lipschitz constant for $\psi^{-1}$. 
 
We have, by \eqref{E:1/3.<.b.lngth.<.2/3} and \eqref{E:lambda.revrs.triangle}, there exists 
$K'' < \infty$ s.t.\  
	\begin{align*}
		\bigl| t_{1} \psi^{-1}( t_{1}^{-1} \tilde{z}_{1}) 
		  - t_{2} \psi^{-1}( t_{2}^{-1} \tilde{z}_{2}) \bigr| 
		  &\leq \bigl| t_{1} \psi^{-1}( t_{1}^{-1} \tilde{z}_{1}) 
		    - t_{1} \psi^{-1}( t_{2}^{-1} \tilde{z}_{2}) \bigr| 
		      + |t_{1} - t_{2}| \bigl|  \bigr| \psi^{-1}( t_{2}^{-1} \tilde{z}_{2}) \bigr| \\
		  &\leq L t_{1} |  t_{1}^{-1} \tilde{z}_{1} -  t_{2}^{-1} \tilde{z}_{2} | 
		    + \tfrac{2}{3} | t_{1} -  t_{2} | \\
		  &\leq L t_{1} |  t_{1}^{-1} \tilde{z}_{1} -  t_{1}^{-1} \tilde{z}_{2} | 
		    + L t_{1} |  t_{1}^{-1} -  t_{2}^{-1} | | \tilde{z}_{2} | 
		      + \tfrac{2}{3} | t_{1} -  t_{2} | \\
		  &\leq L |  \tilde{z}_{1} -  \tilde{z}_{2} | 
		    + 2 L t_{1} |  t_{1}^{-1} -  t_{2}^{-1} | t_{2} + \tfrac{2}{3} | t_{1} -  t_{2} | \\
		  &= L |  \tilde{z}_{1} -  \tilde{z}_{2} | + 2 L |  t_{2} -  t_{1} | 
		    + \tfrac{2}{3} | t_{1} -  t_{2} | \\
		  &\leq K'' \lambda \bigl[ (t_{1}, \tilde{z}_{1}), (t_{2}, \tilde{z}_{2}) \bigr].
	\end{align*}
This completes the proof of the claim that $H$ is bi-Lipschitz w.r.t.\ the Euclidean metric 
on $\RR^{N}$ and $\xi \times \lambda$. Obviously, there is not need for the Lipschitz constants to depend on $R$. Note that, by \eqref{E:finitely.many.A.x.Ztilde}, there is a uniform finite Lipschitz constant for $H_{\mbf{A}, \mbf{Z}}$ 
and $H_{\mbf{A}, \mbf{Z}}^{-1}$ as $(\mbf{A}, \mbf{Z})$ varies.

\emph{Third application of} \eqref{E:general.change.of.variables}: Let
	\begin{equation} \label{E:S3.defn} 
		\mathscr{S}_{3} :=\mathscr{S}_{3, m, \mbf{A}, \mbf{Z}} 
		  := H^{-1} \Bigl( h_{i}^{-1} \bigl[ h_{i}(\mbf{A} \times \tilde{\mbf{Z}}^{\theta_{m}}) 
		    \cap \mathscr{S}_{2,m} \bigr] \Bigr) 
		    \subset \mathscr{X} \subset \RR^{N}.
	\end{equation}
Now, $\mbf{A}$ (see \eqref{E:eta.paramz.A}) and $\mbf{Z}$ (see \eqref{E:psi.paramz.Z}) are manifolds and $H^{-1}$ (see \eqref{E:H.map.defn}) and $h_{i}$ (see part \ref{I:h.hi.invrs.Lip} of definition \ref{D:fibering.by.cones}) are all bi-Lipschitz (in particular they are all inverses of continuous functions). Therefore, by \eqref{E:S2.Borel.finite.meas}, Hardt and Simon \cite[Definition 2.1', p.\ 20]{rHlS86.GMT} and \eqref{E:Lip.magnification.of.Hm} we see that 
	\begin{equation} \label{E:S3.Borel.rect.finite.meas}
		\mathscr{S}_{3} \text{ is Borel, countably $a$-rectifiable, and } 
		  \Hm^{a}(\mathscr{S}_{2}) < \infty.
	\end{equation} 
Let $x = \bigl( e, t(1,b) \bigr) \in \mathscr{S}_{3}$. Then, by \eqref{E:g.commutes.w/.set.ops}, 
$H(x) \in h^{-1}( \mathscr{S}_{2,m}) \subset C[\overline{\clU}]$. Hence, by \eqref{E:hi.invrs.S2.eps.inclusion},  
    \begin{equation} \label{E:S3.subset.Xtheta}
        \bigl( e, t(1,b) \bigr) \in \mathscr{S}_{3} \text{ implies } t \in (\theta_{m}, 1) .   
    \end{equation}

By 
    \begin{itemize}
        \item \eqref{E:dilate.ineq.in.C[U].with.h.invrs}, 
        \item the fact that $H^{-1}$ is bi-Lipschitz, 
        \item \eqref{E:general.change.of.variables} 
        with $\phi := H^{-1}$, 
$S = h_{i}^{-1} \bigl[ h_{i}(\mbf{A} \times \tilde{\mbf{Z}}^{\theta_{m}}) \cap \mathscr{S}_{2,m} \bigr]$, $f = \mathscr{F}_{R}$, and
        \item \eqref{E:suffices.to.prov."dilate.ineq..."}, 
    \end{itemize}
\emph{it suffices to show}
	\begin{equation}   \label{E:dilate.ineq.in.H.invrs.AxZ}
		     \Hm^{a} \bigl[ (H^{-1} \circ \mathscr{F}_{R} \circ H)(\mathscr{S}_{3}) \bigr] \\
		\leq K_{6} R^{p-d+1} \Hm^{a} (\mathscr{S}_{3}).
	\end{equation}
for some $K_{6} < \infty$, etc. 

\emph{We make one final change of variables,} one not making (explicit) use of \eqref{E:general.change.of.variables}. Let $A_{dilate,R}$ and $B_{dilate,R}$ be as \eqref{E:AB.dilate.defn} (with $\threeess$ as in \eqref{E:threeess.appears}), so 
$A_{dilate,R} \asymp 1$ and $B_{dilate,R} \asymp 1$ as $R \downarrow 0$. (See \eqref{E:asymp}.) Write $A := A_{dilate,R}$ and $B := B_{dilate,R}$. Recall $\rho : \D \to [0,1]$ described beginning with \eqref{E:rho.on.W.off.U} and $\eta$ introduced in \eqref{E:eta.paramz.A}. For convenience, 
use $\rho(e)$ as shorthand for $\rho \circ \eta(e)$, but we may still use $\rho(y)$, 
for $y \in \Pf$.

Recall \eqref{E:mathscrX.defn} and \eqref{E:psi.paramz.Z}. Let $h$ be the $h_{i}$ belonging to $\mbf{A}$ and $\mbf{Z}$. So $\mbf{A} \subset \mcl{A}_{i}$ and 
$h : \mbf{A} \times \tilde{\mbf{Z}}^{0} \to C[\mbf{A}]$. (See \eqref{E:bold.Z.tilde.defn}.) Define 
$\zeta : \mathscr{X}_{(B, A+B)} \to \mathscr{X}_{(0, \infty)}$ as follows. 
   \begin{multline} \label{E:zeta.fn.defn}
	\zeta \bigl( e, s(1,b) \bigr) 
	  := \Bigl( e, \bigl(A |v| + B \rho(e) \bigr) |v|^{-1} t \, (1, b) \Bigr)  
	  \in \mathscr{X} (= \mathscr{X}_{(0,1)}), \\
	    \text{ where } \bigl( e,s(1,b) \bigr) \in \mathscr{X}_{(B, A+B)}, \; 
	      t := (s - B)/A \in (0,1), \; \\
	       (y,v) := h \Bigl[ \eta(e), t \bigl(1, \psi(b) \bigr) \Bigr] 
	         = h \circ H \bigl( e, t(1,b) \bigr).
   \end{multline}

\emph{Claim:}
    \begin{equation} \label{E:zeta.is.Lip}
        \zeta \text{ is Lipschitz.}
    \end{equation}
To see this, let $\bigl( e, s(1,b) \bigr) \in \mathscr{X}_{(B, A+B)}$. Let
$(y,v) := h \Bigl[ \eta(e), t \bigl(1, \psi(b) \bigr) \Bigr] = h \circ H \bigl( e, t(1,b) \bigr)$,
 where $t := (s - B)/A$. Thus, $y = \eta(e)$ ($e \in E$). Let $z := \psi(b) \in \mbf{Z}$ 
 ($b \in D$). Let
    \begin{equation}  \label{E:(y,w).from.e.b}
      (y, w) := h \bigl( y, (1,z) \bigr) 
        =  h \Bigl( \eta(e), \bigl( 1,\psi(b) \bigr) \Bigr) \in T_{y} \D , 
    \end{equation}
 i.e.\ like $(y,v)$, but with1 in place of $t$. (This makes sense by remark \ref{R:extending.hi}.)  
Again by remark \ref{R:extending.hi}, $h$ is Lipschitz on 
$\mcl{A}_{i} \times \Bigl[ [0,1] \cdot \bigl( \{1\} \times \msf{L}_{i} \bigr) \Bigr]$. 
By \eqref{E:psi.paramz.Z} and \eqref{E:eta.paramz.A}, $\psi$ and $\eta$ are bi-Lipschitz. It follows from \eqref{E:comp.of.Lips.is.Lip} that $(y,w)$ is Lipschitz in $e$ and $b$. As noted after \eqref{E:varkappa.is.other.proj}, the projection 
$\varkappa$ is Lipschitz. Therefore, $w = \varkappa(y,w)$ is Lipschitz. 
By \eqref{E:rho.is.Lip.on.D}, $\rho$ is Lipschitz. Therefore, by \eqref{E:eta.paramz.A} and \eqref{E:comp.of.Lips.is.Lip}, $\rho \circ \eta$ (sometimes known simply as $\rho$) is Lipschitz.  
Therefore, $\bigl( e, (1,b) \bigr) \mapsto \bigl( y, w, \rho(e) \bigr) 
= \bigl( \eta(e), w, \rho(e) \bigr)$ is Lipschitz in $(e,b) \in E \times D$. 

Since $|v| = t|w|$, 
    \begin{equation}  \label{E:zeta.in.terms.of.w}
        \zeta \bigl( e, s(1,b) \bigr) 
          = \Bigl( e, \bigl(A |v| + B \rho(e) \bigr) |w|^{-1} (1, b) \Bigr) .
    \end{equation}
Now, by definition \ref{D:fibering.by.cones} part \ref{I:h.hi.invrs.Lip}, $h^{-1}$ is Lipschitz. Therefore, by \eqref{E:metric.on.Pf.x.CL}, \eqref{E:lambda.metric.defn},  \eqref{E:xi+.from.2.metrics}, \eqref{E:extended.homogeneity.of.hi}, and \eqref{E:(y,w).from.e.b}, there exists $K < \infty$ s.t.\
    \begin{equation*}
        \tfrac{1}{2} \leq \tfrac{1}{2} \bigl| (1, z) \bigr| 
          = (\xi \times \lambda) \Bigl[ \bigl( y, \tfrac{1}{2}(1,z) \bigr), \bigl( y, (0,0) \bigr) \Bigr]
            \leq K \xi_{+} \Bigl[ \bigl( y , \tfrac{1}{2} w \bigr), (y,0) \Bigr] = \tfrac{1}{2} K |w|.
    \end{equation*}
Thus, $|w| \geq 1/K > 0$ and so is bounded below. The claim \eqref{E:zeta.is.Lip} follows from example \ref{Ex:ratnl.fns.loc.Lip} again.

Recall the definition of $\mathscr{X}_{\ast}$, \eqref{E:mathscrX.defn}. 
Define $g_{dil} : \mathscr{X}_{(0, \infty)} \to \mathscr{X}_{(0, \infty)}$ by
    \begin{equation}  \label{E:g.dil.defn}
        g_{dil} \bigl( e, t(1,b) \bigr) := \bigl( e, (At+B)(1, b) \bigr) \in \RR^{N}, 
          \quad e \in E, t > 0, b \in D.
    \end{equation}
By \eqref{E:theta.m.in.(0,1)}, $\theta_{m} < 1$. Write $\theta = \theta_{m}$. 
Note that $g_{dil}(\mathscr{X}_{\theta}) \subset \mathscr{X}_{(A \theta+B, A+B)}$. But by \eqref{E:S3.subset.Xtheta}, 
$\mathscr{S}_{3} \subset \mathscr{X}_{\theta}$ so 
    \begin{equation}  \label{E:g.dil(S3).subset.X.(A.theta+B,A+B)}
      g_{dil}(\mathscr{S}_{3}) \subset \mathscr{X}_{(A \theta+B, A+B)} .
    \end{equation}

Let $y \in \clU$, $z \in \msf{L}$, and $t \in (0,1)$. 
Write $(y, v) := h \bigl( y, t(1, z) \bigr)$. By \eqref{E:h(y,(1,z)).notin.C[P]}, 
$(y,u) := h \bigl( y, (1, z) \bigr) \notin C[\overline{\clU}]$. Hence, by \eqref{E:eps.=.2}, 
we have $|u| > 4$.
By homogeneity of $h$, we have $(y,v) = t \, h \bigl( y, (1, z) \bigr) = (y, tu)$. 
Thus, $|v| = t |u| > 4 t$ so $1/4 > |v|^{-1} t$. By the hypotheses of lemma \ref{L:dilate.R.-(d-p-1)} (which, as you may recall, is what we're trying to prove here), we have
    \begin{equation} \label{E:0<Rm<threeess/3}
      0 < R_{m} < \threeess/3 .
    \end{equation}
Thus the inequalities \eqref{E:R.<.threeess/3.A.B.ineqs} hold.
Recall $\theta = \theta_{m}$ as in \eqref{E:hi.invrs.theta.m}. Thus, by \eqref{E:rho.in.[0,1]}, \eqref{E:AB.dilate.defn}, and, we have
    \begin{multline} \label{E:Av+B.rho.v.invrs.t.range}
        \bigl( A |v| + B \rho(y) \bigr)  |v|^{-1} t \leq A t + B \rho(y)/4 < A + 2(1 - A) /2 = 1. \\
          \text{ If } t \in (\theta, 1) \text{ then } \theta / 2 < A \theta 
            < \bigl( A |v| + B \rho(y) \bigr)  |v|^{-1} t < 1 .
    \end{multline}
 
Suppose $s \in (A \theta+B, A+B)$. Then $t := (s-B)/A \in (\theta, 1)$. Therefore, by \eqref{E:zeta.fn.defn} and \eqref{E:Av+B.rho.v.invrs.t.range}, we have 
    \begin{equation}  \label{E:zeta.(A.theta+B,A+B).subset.X.theta/2}
      \zeta(\mathscr{X}_{(A \theta+B, A+B)}) \subset \mathscr{X}_{\theta/2} 
     \end{equation}   
  
Recall the definition, \eqref{E:H.map.defn}, of $H$ and the formula, \eqref{E:H.invrs.formla}, for $H^{-1}$. Recall the definition, \eqref{E:F.i.Rm.defn}, of $\mathscr{F}_{R}$ and the definition, \eqref{E:S3.defn}, of $\mathscr{S}_{3}$. \emph{Claim:} The following commutes.
     \begin{equation}  \label{E:gdil.zeta.CD}
          \begin{CD}
            \mathscr{S}_{3} @>{H^{-1} \circ \mathscr{F}_{R} \circ H}>> \mathscr{X}_{\theta/2} \\
    	      @V{g_{dil}}VV @ |  \\ 
    	   \mathscr{X}_{(A \theta+B, A+B)} @>>{\zeta}> \mathscr{X}_{\theta/2} .
          \end{CD}
    \end{equation}
By \eqref{E:S3.subset.Xtheta}, 
    \begin{equation}  \label{E:in.S3,s.in.(theta,1)}
      \text{If } \bigl( e, s(1,b) \bigr) \in \mathscr{S}_{3} 
        \text{ then } s \in (\theta_{m}, 1) .
    \end{equation}
(It follows from \eqref{E:g.dil.defn}, \eqref{E:in.S3,s.in.(theta,1)}, and \eqref{E:R.<.threeess/3.A.B.ineqs} ($0 < R < \threeess/3$ by assumption in lemma \ref{L:dilate.R.-(d-p-1)}), $g_{dil}$ is Lipschitz on $\mathscr{S}_{3}$.)

By \eqref{E:S3.defn}, \eqref{E:S.2m.subset.C2.less.F}, and \eqref{E:Cs.defn}, 
    \begin{equation}  \label{E:h.H.S3.subset.C2}
      (y,v) := h \circ H \bigl( e, s(1,b) \bigr) \in \mathscr{S}_{2} \subset C_{2}. 
        \text{ Therefore, } |v| < 2 \rho(y) .
    \end{equation}
In particular, by \eqref{E:H.map.defn}, 
     \begin{equation} \label{E:y.z.and.v}
      y = \eta(e) \text{ and, if } z := \psi(b), 
        \text{ then } (y,v) := h \bigl( y, s(1,z) \bigr) .
    \end{equation}

By \eqref{E:g.dil(S3).subset.X.(A.theta+B,A+B)}, we have 
$g_{dil} \bigl( e, s(1,b) \bigr) \in \mathscr{X}_{(A \theta+B, A+B)}$. 
Therefore, by \eqref{E:zeta.(A.theta+B,A+B).subset.X.theta/2}, $\zeta \circ g_{dil}\bigl( e, s(1,b) \bigr) \in \mathscr{X}_{\theta/2}$. Specifically, since $\bigl[ (As+B) - B \bigr]/A = s$, by \eqref{E:g.dil.defn} and \eqref{E:zeta.fn.defn}, 
    \begin{align} \label{E:zeta.circ.gdil}
        \zeta \circ g_{dil}\bigl( e, s(1,b) \bigr) &= \zeta \bigl( e, (As+B)(1,b) \bigr) \notag \\
          &= \Bigl( e, \bigl( A |v| + B \rho(y) \bigr) |v|^{-1} s(1, b) \Bigr) . 
    \end{align}
That is what you get if you follow the low road in \eqref{E:gdil.zeta.CD}. 

Now we take the high road in \eqref{E:gdil.zeta.CD}, \emph{viz.} we evaluate 
$(H^{-1} \circ \mathscr{F}_{R} \circ H) \bigl( e, s(1,b) \bigr)$. 
As remarked, $s \in (\theta, 1)$. \emph{A fortiori}, by \eqref{E:rho.in.[0,1]}, 
$s \in \bigl( \rho(y) \theta, 1 \bigr)$. Hence, by \eqref{E:threeess.appears}, 
we have 
    \begin{equation*}
      |v| \geq R \rho(y)/\threeess .
    \end{equation*}
Therefore, by\eqref{E:F.i.Rm.defn}, \eqref{E:H.map.defn}, and \eqref{E:h.H.S3.subset.C2},  
we have
    \begin{equation*} 
            (H^{-1} \circ \mathscr{F}_{R} \circ H) \bigl( e, s(1,b) \bigr) 
              = (H^{-1} \circ h^{-1} \circ F) \Bigl[ h \circ H \bigl( e, s(1,b) \bigr) \Bigr] \notag 
                = H^{-1} \circ h^{-1} \circ F(y,v) .
    \end{equation*}
We know $|v| > R \rho(y)/\threeess$. And by \eqref{E:h.H.S3.subset.C2} again, 
$|v| < 2 \rho(y)$. Therefore, by \eqref{E:F.on.C[Pf].defn},  
    \begin{equation*}
      (H^{-1} \circ \mathscr{F}_{R} \circ H) \bigl( e, s(1,b) \bigr)
        = H^{-1} \circ h^{-1} \Bigl( y, \bigl( A |v| + B \rho(e) \bigr) |v|^{-1} v \Bigr) .
    \end{equation*}
But $h^{-1}$ is homogeneous, by \eqref{E:homogeneity.of.hi}. Therefore, by \eqref{E:y.z.and.v}, \eqref{E:h.H.S3.subset.C2}, and \eqref{E:H.invrs.formla}, 
    \begin{multline*}
      (H^{-1} \circ \mathscr{F}_{R} \circ H) \bigl( e, s(1,b) \bigr) 
        = H^{-1} \Bigl( y,  \bigl[ A |v| + B \rho(e) \bigr] |v|^{-1} s (1, z) \Bigr) \\
          = \Bigl( e, \bigl (A |v| + B \rho(y) \bigr) |v|^{-1} s(1,b) \Bigr) .
    \end{multline*}
This is the same as what we got in \eqref{E:zeta.circ.gdil}. This proves the claim that 
\eqref{E:gdil.zeta.CD} commutes and, incidentally, 
that $H^{-1} \circ \mathscr{F}_{R} \circ H (\mathscr{S}_{3}) \subset \mathscr{X}_{\theta/2}$.

By \eqref{E:zeta.is.Lip}, $\zeta$ has a Lipschitz constant $K < \infty$. Then, by \eqref{E:gdil.zeta.CD} and \eqref{E:Lip.magnification.of.Hm}, we have 
    \begin{equation*}
        \Hm^{a} \bigl[ (H^{-1} \circ \mathscr{F}_{R} \circ H)(\mathscr{S}_{3}) \bigr] 
          = \Hm^{a} \bigl[ (\zeta \circ g_{dil}) (\mathscr{S}_{3}) \bigr] 
            \leq K^{a} \Hm^{a} \bigl[ g_{dil} (\mathscr{S}_{3}) \bigr].
    \end{equation*}
Therefore, by \eqref{E:dilate.ineq.in.H.invrs.AxZ}, to prove lemma \ref{L:dilate.R.-(d-p-1)}, 
    \begin{equation}  \label{E:Hm.a.gdil.S3.R-min.ineq}
      \text{It suffices to show that for some } K_{8} < \infty, \;
        \Hm^{a} \bigl[ g_{dil} (\mathscr{S}_{3}) \bigr]  
          \leq K_{8} R^{p-d+1} \Hm^{a} (\mathscr{S}_{3}).
    \end{equation}

\vspace{0.20in}
	
\subsubsection{Bounding an integral} \label{SSS:bound.integral}
So it remains to prove \eqref{E:Hm.a.gdil.S3.R-min.ineq}. Recall \eqref{E:S3.Borel.rect.finite.meas}. And obviously, by \eqref{E:R.<.threeess/3.A.B.ineqs}, $g_{dil}$ defined in \eqref{E:g.dil.defn} is injective. Therefore, by the ``area formula'' 
(Hardt and Simon \cite[1.8 p.\ 13, p.\ 14, and p.\ 27]{rHlS86.GMT}, Federer \cite[3.2.3, p.\ 243 and 3.2.46, p.\ 282]{hF69}, we give details presently), we have
	\begin{equation}  \label{E:integral.over.S3}
		\Hm^{a} \bigl[ g_{dil}(\mathscr{S}_{3}) \bigr]  
		  = \int_{\mathscr{S}_{3}} J^{\mathscr{S}_{3}} g_{dil} (e,t,b) \; 
		    \Hm^{a} (de \, dt \, db),
	\end{equation}
where $\Hm^{a}$ is computed based on the Euclidean metric on $\RR^{N}$, where 
$N := \ell+q+1$. (See \eqref{E:eta.paramz.A} and \eqref{E:psi.paramz.Z}.) Here, $J^{\mathscr{S}_{3}} g_{dil}$ is defined in Hardt and Simon \cite[pp.\ 13, 27]{rHlS86.GMT}). To bound $J^{\mathscr{S}_{3}} g_{dil}$ we use the following lemma, proved in appendix \ref{Chptr:misc.proofs}. First, note this elementary fact. Let $V$ and $W$ be inner product spaces let $L : V \to W$ be linear, and let $L^{adj}$ be the adjoint of $L$. Then 
$L^{adj} L : V \to V$ is self-adjoint so its eigenvalues are real (Stoll and Wong \cite[Theorem 4.1, p.\ 207]{rrSetW68.LinearAlgebra}). Moreover, its eigenvalues are non-negative. To see this, let $v \in V$. Then
	\begin{equation}  \label{E:L.is.noneg.def}
		\bigl\langle (L^{adj} L) v, v \bigr\rangle_{V} 
		   = \langle L v, L v \rangle_{W} \geq 0,
	\end{equation}
where $\langle \cdot, \cdot \rangle_{V}$ and $\langle \cdot, \cdot \rangle_{W}$ are the inner products on $V$ and $W$, resp. The proof of the following is in appendix \ref{Chptr:misc.proofs}.
 	\begin{lemma}   \label{L:rect.set.Jake.bounds}
Let $M$ be an $m$-dimensional $C^{\infty}$-manifold imbedded in some $\RR^{n_{1}}$, $n_{1} \geq m$. Put on $M$ the Riemannian metric it inherits from $\RR^{n_{1}}$. 
Let $n_{2} \geq m$ be an integer and suppose 
$g = (g_{1}, \ldots, g_{n_{2}}) : M \to \RR^{n_{ 2}}$ is continuously differentiable. Let $X$ be an imbedded countably $\Xdim$-rectifiable subset of $M$, 
where $0 \leq \Xdim \leq m$ is an integer. Then the following holds for $\Hm^{r}$-almost all $x \in X$: The linear map $d^{M} g_{x} : T_{x} M \to \RR^{n_{2}}$ exists (Hardt and Simon \cite[p.\ 13]{rHlS86.GMT}). 
 For $x \in M$, let $(d^{M}g_{x})^{adj} : T_{g(x)} \RR^{n_{2}} \to T_{x} M$ be the adjoint of $d^{M}g_{x}$ (defined in terms of the Riemannian metrics on $ \RR^{n_{2}}$ and $M$, the latter inherited from $\RR^{n_{1}}$; of course $T_{g(x)} \RR^{n_{2}} \isomto \RR^{n_{2}}$).
  Let $\nu_{1}^{2}(x) \geq \cdots \geq \nu_{m}^{2}(x) \geq 0$, 
  be the eigenvalues of the transformation 
  $(d^{M}g_{x})^{adj} \circ (d^{M}g_{x}) : T_{x} M \to T_{x} M$, where $\nu_{j}(x) \geq 0$ for all $j$. Then
     \[
        \nu_{m-\Xdim+1}(x) \cdots \nu_{m}(x) \leq J^{X}g(x) 
          \leq \nu_{1}(x) \cdots \nu_{\Xdim}(x), \quad \Hm^{\Xdim} 
            \text{-almost everywhere}.
     \]
	 \end{lemma}

We apply the lemma with: $g := g_{dil}$ (defined in \eqref{E:g.dil.defn}), 
$M := \mathscr{X} = \mathscr{X}_{(0, 1)}$ (see \eqref{E:mathscrX.defn}), and 
$X := \mathscr{S}_{3}$ (see \eqref{E:S3.defn}, \eqref{E:S3.Borel.rect.finite.meas}, and \eqref{E:S3.subset.Xtheta}). So we need to compute 
$d^{\mathscr{X}} g_{dil, x}$ ($x \in \mathscr{X}$) and its adjoint. But, as remarked just after \eqref{E:mathscrX.defn}, $\mathscr{X}$ is an \emph{open} subset of $\RR^{N}$ 
(recall $N := \ell+q+1$). Therefore, $d^{\mathscr{X}} g_{dil, \cdot}$ is just the differential, 
$D g_{dil}$. 

Now, by \eqref{E:g.dil.defn}, $g_{dil}$ is not really defined on $(e,t,b)$-space. It takes an argument $x \in \RR^{N}$ of the form:  
    \begin{equation*}
      x = (e, t, w) = \bigl( e, t(1, b) \bigr) \in \mathscr{X} ,
    \end{equation*} 
where $b \in D$. So $w = t b \in \RR^{q}$. (We will treat $b$ as a $1 \times q$ row vector.) Let $x$ be an arbitrary vector of that form. So we are interested in  
    \begin{equation}  \label{E:t.in.(0,1)}
      t \in (0,1).
    \end{equation} 
Let $A_{dilate,R}$ and $B_{dilate,R}$ be as \eqref{E:AB.dilate.defn}, 
so $A_{dilate,R} \asymp 1$ and $B_{dilate,R} \asymp 1$ as 
$R \downarrow 0$. Write $A := A_{dilate,R}$ and $B := B_{dilate,R}$. 
Hence, $g_{dil}(x) = \bigl( e, At+B, (A + B t^{-1}) w \bigr)$. Write 
$g_{dil} = (g_{dil,1}, \ldots, g_{dil,N})$. Use ``${}^{T}$'' to indicate transposition. Observe that  $(\partial/\partial t) (A + B t^{-1}) w^{T} = -B t^{-2}  w^{T}= -B t^{-1} b^{T}$. Then the matrix of the derivative (or Jacobian) matrix, $D g_{dil}(x) := (\partial g_{dil,i}/\partial x_{j} )$, 
of $g_{dil \ast}$ is, 
	\begin{equation}  \label{E:matrix.of.Dg.dil}
		Dg_{dil}(x) =
			\begin{pmatrix}
				I_{\ell} & 0^{\ell \times 1} & 0^{\ell \times q} \\
				0^{1 \times \ell} & A^{1 \times 1} & 0^{1 \times q} \\  
				0^{q \times \ell} & (-B t^{-1} b^{T})^{q \times 1} & (A + B t^{-1}) I_{q}
			\end{pmatrix}.
	\end{equation}
Here, $I_{q}$ is the $q \times q$ identity matrix. Since $g_{dil}$ maps an open subset 
of $\RR^{N}$ into $\RR^{N}$, we have that the matrix of the adjoint of $d^{\mathscr{X}}$ is just the transpose $Dg_{dil}(x)^{T}$ and 
	\begin{equation*}  
		Dg_{dil}(x)^{T} \, Dg_{dil}(x) =
			\begin{pmatrix}
				I_{\ell} & 0^{\ell \times 1} & 0^{\ell \times q} \\
				0^{1 \times \ell} & A^{2} + B^{2} t^{-2} |b|^{2} & 
				  -B t^{-1} (A + B t^{-1}) b^{1 \times q} \\  
				0^{q \times \ell} & \bigl( -B t^{-1} (A + B t^{-1}) b^{T} \bigr)^{q \times 1}
				 & (A + B t^{-1})^{2} I_{q}
			\end{pmatrix}.
	\end{equation*}

To apply lemma \ref{L:rect.set.Jake.bounds}, we need the eigenvalues 
of $Dg_{dil}(x)^{T} \, Dg_{dil}(x)$. 
Any vector of the form $(v^{1 \times \ell}, 0^{1 \times 1}, 0^{1 \times q})$ is an eigenvector with eigenvalue 1. If $v^{1 \times q}$ is orthogonal to $b$ then
$(0^{1 \times \ell}, 0^{1 \times 1}, v^{1 \times q})$ is an eigenvector with eigenvalue 
$(A + B t^{-1})^{2} \asymp t^{-2}$ as $t \downarrow 0$. Hence, we have calculated 
	\begin{multline}  \label{E:N-2.eigvals.of.DTD}
		N-2 \text{ eigenvalues of } Dg_{dil}(x)^{T} \, Dg_{dil}(x): \\
		  \text{ 1 with multiplicity } \ell \text{ and } (A + B t^{-1})^{2} \asymp t^{-2} 
		    \text{ with multiplicity } q-1, \\
		       t \in (0,1) \text{ and } R \in (0, \threeess/3), 
                          \text{ uniformly in } e \in E, b \in D, \\
                            \text{ where } z = \bigl( e,t(1,b) \bigr) \in \mathscr{X}_{(0, 1)} .
	\end{multline}
Notice that $(0,0,b)$ is not an eigenvector. 

The remaining two eigenvectors have to span a subspace, $V$, of $\RR^{N}$ orthogonal to the eigenvectors described in the last paragraph. It suffices to study the eigenvalues of the 
    \begin{multline}  \label{E:matrix.M.defn}
       \bigl(  \text{Lower quadrant of } Dg_{dil}(x)^{T} \, Dg_{dil}(x) \bigr) 
         = t^{-2} M, \text{ where } \\
           M^{(q+1) \times (q+1)} := M(x) := 
			\begin{pmatrix}
				A^{2}  t^{2} + B^{2} |b|^{2} & -B (A  t + B) b^{1 \times q} \\  
				\bigl( -B (A  t + B) b^{T} \bigr)^{q \times 1} & (A t + B)^{2} I_{q}
			\end{pmatrix}.
    \end{multline}
Thus, keeping in mind the $t^{-2}$, we only need to find the eigenvalues of $M$ corresponding to eigenvectors of the form 
    \begin{equation*}
      (\alpha^{1 \times 1}, b) \text{ with } \alpha \neq 0 ,
    \end{equation*} 
because we have already counted the eigenvectors that are all 0 except the last $q$ coordinates, which constitute a vector orthogonal to $b$. 

Let 
    \begin{equation*}
      \beta := |b|^{2} .
    \end{equation*}  
so, by \eqref{E:1/3.<.b.lngth.<.2/3}, 
    \begin{equation} \label{E:1/9.<.beta.<.4/9}
      1/9 < \beta<  4/9 \text{ for every } b \in D.
    \end{equation}
Now let
    \begin{multline}  \label{E:C(t)}
        C := C(t) := - \frac{(1-\beta) B^{2} + 2 ABt}{At + B} = - B - B \frac{At - \beta B}{At+B} \\
          = -B - B \frac{(At + B) - B(1+\beta)}{At+B}
            = -2B + \frac{B^{2} (1+\beta) }{At+B} .
    \end{multline}
By \eqref{E:R.<.threeess/3.A.B.ineqs} and \eqref{E:1/9.<.beta.<.4/9},
    \begin{equation*}
      C(t) \text{ is bounded and bounded away from 0 in } t \geq 0 . 
    \end{equation*}
For future reference, 
    \begin{align}  \label{E:C.properties} 
        C(0) &= -(1-\beta) B < 0, \notag \\
        C'(t) &= \frac{-A B^{2} (1+\beta) }{ (At+B)^{2} } \text{ so } C'(0) = -A(1 + \beta), \\
        C''(t) &= \frac{ 2A^{2}B^{2}(1+\beta) }{ (At+B)^{3} } \text{ so } C''(0) 
          = \frac{ 2A^{2}(1+\beta) }{ B }, 
          \text{ and } \notag \\
        C^{(3)}(t) &= \frac{ -6 A^{3}B^{2}(1+\beta) }{ (At+B)^{4} } \text{ so } C^{(3)}(0)
          = -\frac{ 6A^{3}(1+\beta) }{ B^{2} } \notag.
    \end{align}
By \eqref{E:R.<.threeess/3.A.B.ineqs} and \eqref{E:1/3.<.b.lngth.<.2/3}, $C^{j}(0) \neq 0$ 
and $C^{(j)}(t)$ is bounded in $t \geq 0$ ($j = 0, 1, 2, 3$). 

From \eqref{E:C(t)}, we have    
    \begin{equation*}
      C(t)^{2} = 4 B^{2} - \frac{4 B^{3} (1+\beta) }
        {At+B} + \frac{B^{4} (1+\beta)^{2} }{(At+B)^{2}}
          = 4 B^{2} + B^{3} (1+\beta) \left[ \frac{B (1+\beta) }
            {(At+B)^{2}} -  \frac{4}{At+B} \right] .
    \end{equation*}
Therefore, with some drudgery we derive a bunch of complicated formulas. To do this it may help to use ``Leibnitz's formula'': Given two functions, $u$, $v$ we have 
    \begin{equation*}
      \frac{d^{n}}{dt^{n}} (uv) 
        = \sum_{k=0}^{n} \binom{n}{k} u^{(k)} v^{(n-k_)}, \; n = 0, 1, \ldots . 
    \end{equation*}
Applying this with $u = v = C$ it is not too awful to compute:
    \begin{align}  \label{E:C.sqrd.properties} 
        C^{2}(0) &= B^{2} (1-\beta)^{2} > 0 \notag \\
          \tfrac{d}{dt} C^{2}(t) &= 2AB^{3} (1+\beta) \left[ -\frac{B (1+\beta) }{(At+B)^{3}}  
            + \frac{2}{(At+B)^{2}} \right], \notag \\
              & \qquad \text{ so }  \tfrac{d}{dt} C^{2}(t) \restriction_{t=0} 
                = 2AB (1 - \beta^{2}),  \notag \\
        \tfrac{d^{2}}{dt^{2}} C^{2}(t) 
          &= 2A^{2}B^{3} (1+\beta) \left[ \frac{3B (1+\beta) }{(At+B)^{4}}  
            - \frac{4}{(At+B)^{3}} \right]  \\
              & \qquad \text{ so } \tfrac{d^{2}}{dt^{2}} C^{2}(t) \restriction_{t=0} 
                = 2A^{2} (3 \beta - 1)(\beta + 1)  
            \text{ and } \notag \\
        \tfrac{d^{3}}{dt^{3}} C^{2}(t)(t) 
          &= 24A^{3}B^{3} (1+\beta) \frac{At-B \beta}{(At+B)^{5}} 
              \left[ -\frac{B (1+\beta) }{(At+B)^{5}} 
                + \frac{1}{(At+B)^{4}} \right] \notag \\
          &= 24A^{3}B^{3} (1+\beta) \frac{At-B \beta}{(At+B)^{5}} \notag \\
          & \qquad \text{ so } \tfrac{d^{3}}{dt^{3}} C^{2}(t) \restriction_{t=0} 
            = - 24 A^{3} B^{-1} \beta (1+\beta) \notag.
    \end{align}

Recall that, by assumption in lemma \ref{E:dilate.R.-(d-p-1)}, 
$0 < R_{m} < \threeess/3$. By \eqref{E:R.<.threeess/3.A.B.ineqs} and \eqref{E:1/9.<.beta.<.4/9}, the expression 
$24A^{3}B^{3} (1+\beta)$ in $\tfrac{d^{3}}{dt^{3}} C^{2}(t)(t)$ is bounded and bounded away from 0 as a function of $A$, $B$, and $\beta$. Suppose $t \in (0, 1/9)$. Then 
$\tfrac{At-B \beta}{(At+B)^{5}}$ is negative, finite, and bounded away from 0 as a function of $A$, $B$, and $t \in (0,1/9)$. Thus, by \eqref{E:C.sqrd.properties},  
    \begin{multline}  \label{E:C.sqrd.3rd.deriv.bnded}
      \tfrac{d^{3}}{dt^{3}} C^{2}(t) 
        \text{ is bounded and bounded away from 0 if } t \in (0, 1/9) . \\
        \text{Conseqently, }
        \tfrac{d^{3}}{dt^{3}} C^{2}(t) \asymp 1 \text{ as } t \downarrow 0. . 
    \end{multline}

Note further that 
    \begin{equation} \label{E:now.I.know.my.ABCs}
        \bigl( At + B + C(t) \bigr) (At + B) = A^{2} t^{2} + B^{2} \beta.
    \end{equation}
Then it is easily seen from \eqref{E:matrix.M.defn}, that 
    \begin{multline} \label{E:M'.matrix.defn}
        M = (At+B) M', \text{ where } \\
          M'^{(q+1) \times (q+1)} := M'(x) := 
			\begin{pmatrix}
				A t + B + C & -B b^{1 \times q} \\  
				\bigl( -B b^{T} \bigr)^{q \times 1} & (A t + B) I_{q}
			\end{pmatrix}.
    \end{multline}
Hence, we only need to find the eigenvalues of $M'$ corresponding to eigenvectors of the form $(\alpha^{1 \times 1}, b)$. We already observed that $\alpha$ cannot be 0.  
Hence, it suffices to find $\gamma = \gamma(t) \in \RR$ s.t.\ $(1, \gamma b)$ is an eigenvalue of $M'$. Denote the corresponding eigenvalue by $\nu^{2}$, with $\nu \geq 0$. 
(Since $Dg_{dil}(x)^{T} \, Dg_{dil}(x)$ is non-negative definite, by \eqref{E:L.is.noneg.def}, so is $M'$.) Recall that $\beta := |b|^{2}$. Thus,
    \begin{equation} \label{E:nu.gamma.M'}
        \nu^{2} 
			\begin{pmatrix}
				1 \\ 
				\gamma b^{T}
			\end{pmatrix}
          = M' 
			\begin{pmatrix}
				1 \\ 
				\gamma b^{T}
			\end{pmatrix}
			=
			\begin{pmatrix}
				A t + B + C  -\gamma B \beta \\  
				-B b^{T} + \gamma (A t + B) b^{T}
			\end{pmatrix}.
    \end{equation}

Observe that $\nu^{2}$ cannot be 0. If it were we would have, from the first row in the preceding matrix 
$\gamma = (A t + B + C)/(B \beta)$ and, hence, from the second row, 
$B = (A t + B + C)(A t + B)/(B \beta)$. Thus, by \eqref{E:now.I.know.my.ABCs}, we have $B^{2} \beta = A^{2} t^{2} + B^{2} \beta$, so $A^{2} t^{2} = 0$. 
But by \eqref{E:R.<.threeess/3.A.B.ineqs} $A \neq 0$, and, by assumption, $t \in (0,1)$ so $A^{2} t^{2} \neq 0$. Contradiction. Thus, $\nu^{2} \neq 0$. From \eqref{E:nu.gamma.M'} we see
    \begin{equation} \label{E:nu.no.bT}
        \nu^{2} = A t + B + C  -\gamma B \beta \text{ and } \nu^{2} \gamma b^{T} 
          = -B b^{T} + \gamma \, (A t + B) b^{T}. 
    \end{equation}
Therefore,
    \begin{equation*}
        \gamma \, (A t + B + C  -\gamma B \beta) = -B + \gamma \, (A t + B), 
          \text{ which is equivalent to } B \beta \gamma^{2} - C \gamma - B = 0.
    \end{equation*}
Hence,
    \begin{equation}  \label{E:gamma.quad.soln}
        \gamma = \frac{ C \pm \sqrt{C^{2} + 4 B^{2} \beta} }{ 2B \beta }.
    \end{equation}
By \eqref{E:C.sqrd.properties}, we have
    \begin{equation}  \label{E:C(0)sqrd+4.Bsqrd.beta}
      C(0)^{2} + 4 B^{2} \beta = (1 + \beta)^{2} B^{2} .
    \end{equation}
Using this in combination with \eqref{E:C.properties}, \eqref{E:R.<.threeess/3.A.B.ineqs}, and \eqref{E:1/9.<.beta.<.4/9}, we get that
    \begin{equation}  \label{E:gamma.limit}
        \gamma(t) \to \frac{ -(1-\beta) \pm (1+\beta) }{ 2 \beta } \neq 0 
          \text{ as } t \downarrow 0.
    \end{equation}

First, replace ``$\pm$'' in \eqref{E:gamma.quad.soln} by ``$-$'' and call the corresponding 
$\gamma$ ``$\gamma_{-}$''. Call the corresponding eigenvector $\nu_{-}^{2}$. By \eqref{E:gamma.limit} we see that  
    \begin{equation}  \label{E:gamma.minus.asymp.1}
        \gamma_{-}(t) \to -1/\beta < 0 \text{ as } t \downarrow 0
    \end{equation}
so $\gamma_{-} < 0$ for $t$ sufficiently small. Plugging this and the first line of \eqref{E:C.properties} into the first part of \eqref{E:nu.no.bT} we see
    \begin{equation}  \label{E:nu.minus.limit}
       \nu_{-}^{2}(t) \to B - (1-\beta) B + B = (1+\beta) B > 0.
    \end{equation}
I.e., $\nu_{-}^{2}(t) \asymp 1$. Recalling the definitions \eqref{E:M'.matrix.defn} and \eqref{E:matrix.M.defn} of $M'$ and $M$, resp., we see that 
    \begin{equation} \label{E:what.nu.minus.means}
      \nu_{-}^{2} \text{ translates into an eigenvalue of } Dg_{dil}(x)^{T} \, Dg_{dil}(x)
        \text{ that } \asymp t^{-2} \text{ as } t \downarrow 0.
    \end{equation}

Now replace ``$\pm$'' in \eqref{E:gamma.quad.soln} by ``$+$'' and call the corresponding 
$\gamma$ ``$\gamma_{+}$''. Denote the corresponding eigenvalue 
by $\nu_{+}^{2} = \nu_{+}^{2}(t)$. From \eqref{E:gamma.limit} we see that 
$\gamma_{+} \to 1$ as $t \downarrow 0$. Combining this with \eqref{E:C.properties} and \eqref{E:nu.no.bT} we see that $\nu_{+}^{2}(t) \to 0$ as $t \downarrow 0$. We show that, more precisely, 
    \begin{equation} \label{E:nu+.asymptotics}
        \nu_{+}^{2}(t) \asymp t^{2} \text{ as } t \downarrow 0.
    \end{equation}

It follows from the Taylor approximation (Apostol \cite[Theorem 5-14, p.\ 96]{tmA57.Apostol}), \eqref{E:C.sqrd.3rd.deriv.bnded}, and \eqref{E:C.sqrd.properties} that 
    \begin{multline*}
        C^{2}(t) = C(0)^{2} + \left( \frac{d}{ds} C^{2}(s) \restriction_{s=0} \right) t 
          + \left( \frac{1}{2} \frac{d^{2}}{ds^{2}} C^{2}(s) \restriction_{s=0} \right) t^{2} 
            + o(t^{2}) \\                 
              = (1-\beta)^{2} B^{2} + 2(1-\beta^{2}) A B t + A^{2} 
                (3 \beta - 1)(\beta + 1) t^{2} + o(t^{2}) 
                  \text{ as } t \downarrow 0.   
    \end{multline*}
Here, $o(t^{2})$ (Landau ``little o'' notation, de Bruijn \cite[Section 1.3]{ngdeB81.AsympMthdsAnlys}) only depends only on $A$, $B$, $\beta$, and $t$. Observe that $(1-\beta)^{2} B^{2} + 4 B^{2} \beta = (1 + \beta)^{2} B^{2}$. Examining \eqref{E:gamma.quad.soln} and using \eqref{E:C(0)sqrd+4.Bsqrd.beta}, we see that to compute $\gamma(t)$ we thus need to evaluate
    \begin{equation}  \label{E:C.sqrd(t)+4B.sqrd.beta}
        C^{2}(t) + 4 B^{2} \beta 
          = (1 + \beta)^{2} B^{2} + 2(1-\beta^{2}) A B t 
            + A^{2} (3 \beta - 1)(\beta + 1) t^{2} + o(t^{2}).
    \end{equation}
Therefore, to use ``big O'' notation (de Bruijn \cite[Section 1.2]{ngdeB81.AsympMthdsAnlys}), $C^{2}(t) - C^{2}(0) = O(t)$.
 
Recall \eqref{E:C(0)sqrd+4.Bsqrd.beta} and let 
    \begin{equation*}
     \Delta(t) := \bigl[ C^{2}(t) + 4 B^{2} \beta \bigr] - \bigl[ C^{2}(0) + 4 B^{2} \beta \bigr]
       = \bigl[ C^{2}(t) + 4 B^{2} \beta \bigr] - (1+\beta)^{2} B^{2} .
    \end{equation*}
Then from \eqref{E:C.sqrd(t)+4B.sqrd.beta},
    \begin{align}  \label{E:Delta(t),Delta(t).sqrd}
         \Delta(t) &= 2(1-\beta^{2}) A B t + A^{2} (3 \beta - 1)(\beta + 1) t^{2} 
           + o(t^{2}) = O(t) \\
         \Delta(t)^{2} &= 4 (1-\beta^{2})^{2} A^{2} B^{2} t^{2} + o(t^{2}) \notag .
    \end{align}
Hence, $\Delta(t)^{3} = o(t^{2})$. We have observed that $C(t)$ is bounded in $t \geq 0$. 
Moreover, by \eqref{E:R.<.threeess/3.A.B.ineqs}, $B > 0$ is bounded and bounded away from 0. Therefore, the third derivative of the square root function is bounded on the interval with endpoints $C^{2}(t) + 4 B^{2} \beta$ and $(1+\beta)^{2} B^{2}$ for all $t \geq 0$. 
Therefore, expanding the square root function in a Taylor expansion about 
$C^{2}(0) + 4 B^{2} = (1 + \beta)^{2} B^{2}$ (see \eqref{E:C(0)sqrd+4.Bsqrd.beta}) and applying \eqref{E:Delta(t),Delta(t).sqrd}, we get,  
    \begin{align}  \label{E:Csqrd(t)+4.Bsqrd.beta.expansion} 
       \sqrt{ C^{2}(t) + 4 B^{2} \beta }
          &= \sqrt{(1+\beta)^{2} B^{2}} 
           + \frac{1}{2\sqrt{(1+\beta)^{2} B^{2}}} \Delta(t) 
             - \frac{1}{ 2 \times 4 \bigl( (1+\beta)^{2} B^{2} \bigr)^{3/2} } \Delta(t)^{2} + o(t^{2}) 
               \notag \\
         &=  (1+\beta) B +  \frac{1}{2 (1+\beta) B} \Delta(t) 
           - \frac{1}{ 8 \bigl( (1+\beta) B \bigr)^{3} } \Delta(t)^{2}
           + o(t^{2}) \notag \\
         &=  (1+\beta) B + \left( \frac{2(1-\beta^{2}) A B }{2(1+\beta) B} t 
           + \frac{A^{2} (3 \beta - 1)(\beta + 1)}{2 (1+\beta) B} t^{2} \right) \\
         & \qquad \qquad - \frac{4 (1-\beta^{2})^{2} A^{2} B^{2} }
           { 8 \bigl( (1+\beta) B \bigr)^{3} } t^{2} + o(t^{2}) \notag \\
         &=  (1+\beta) B + (1-\beta) A  t 
           + \frac{A^{2} (3 \beta - 1)(\beta + 1)}{2 (1+\beta) B} t^{2}
           - \frac{(1-\beta)^{2} A^{2} }{ 2 (1+\beta) B } t^{2}
           + o(t^{2}) \notag \\
         &=  (1+\beta) B + (1-\beta) A  t 
           + \frac{A^{2}}{(1+\beta) B } ( \beta^{2} + 2 \beta - 1 ) t^{2}
           + o(t^{2}) .  \notag    
     \end{align}
    
Using \eqref{E:C.properties}, expand $C(t)$ in a second order Taylor expansion about $t=0$:
\linebreak 
$C(t) = -(1-\beta) B - A(1 + \beta) t +  \frac{ A^{2}(1+\beta) }{ B } t^{2} + o(t^{2})$. Therefore, from \eqref{E:gamma.quad.soln} and \eqref{E:Csqrd(t)+4.Bsqrd.beta.expansion}, 
    \begin{align*}
        2 B \beta \gamma_{+}(t) &= -(1-\beta) B - A(1 + \beta) t 
          +  \frac{ A^{2}(1+\beta) }{ B } t^{2} + (1+\beta)B \\
           & \qquad \qquad + (1-\beta) A t + A^{2} \frac{\beta^{2} 
             + 2 \beta - 1}{(1+\beta) B} t^{2} + o(t^{2}) \\
           &= 2 \beta B - 2 A \beta t + 2 A^{2} \beta \frac{ \beta + 2 }{(1+\beta) B} t^{2} + o(t^{2}) .
    \end{align*}

Hence, by the preceding, \eqref{E:nu.no.bT} and \eqref{E:now.I.know.my.ABCs},
    \begin{align*}
        2 (At+B) \nu_{+}^{2}(t) &=  2 A^{2} t^{2} + 2 B^{2} \beta 
          - 2 (At+B) \gamma_{+}(t) B \beta \\
           &= 2 A^{2} t^{2} + 2 B^{2} \beta - (At+B) \left[ 2 \beta (B - A t) + 
               2 A^{2} \beta \frac{ \beta + 2 }{B(1+\beta)} t^{2} \right] + o(t^{2}) \\
           &= 2 A^{2} t^{2} + 2 B^{2} \beta 
             - 2 \beta (B^{2} - A^{2} t^{2}) - 2 A^{2} \beta \frac{ \beta + 2 }{1+\beta} t^{2} 
               + o(t^{2}) \\
           &= 2 A^{2} t^{2} 
             + 2 \beta A^{2} t^{2} - 2 A^{2} \beta \frac{ \beta + 2 }{1+\beta} t^{2} + o(t^{2}) \\
           &= 2 A^{2} \left[ 1+\beta - \beta \frac{ \beta + 2 }{1+\beta} \right] t^{2} + o(t^{2}) \\
           &= \frac{ 2 A^{2}}{1+\beta} t^{2} + o(t^{2})
    \end{align*}
Hence, by \eqref{E:1/9.<.beta.<.4/9} and \eqref{E:R.<.threeess/3.A.B.ineqs}, we have 
that $\nu_{+}^{2}(t) \asymp t^{2}$ as $t \downarrow 0$. Hence, recalling the definitions \eqref{E:M'.matrix.defn} and \eqref{E:matrix.M.defn} of $M'$ and $M$, resp., we see that 
$\nu_{+}^{2}$ translates into an eigenvalue $\asymp 1$ as $t \downarrow 0$ 
of $Dg_{dil}(x)^{T} \, Dg_{dil}(x)$.

Combining this with \eqref{E:N-2.eigvals.of.DTD} and \eqref{E:what.nu.minus.means},
we conclude that these eigenvalues of $J$ have the following bounds:
	\begin{multline}  \label{E:asymptotics.of.eigvals.of.J}
		\ell + 1 \text{ eigenvalues are bounded and } q \text{ eigenvalues are } 
		  < K'' t^{-2}, 
		  \text{ for some } K'' < \infty, \\
		    \text{ uniformly in } e \in E, b \in D, \text{ where } z = \bigl( e,t(1,b) \bigr) 
		      \in \mathscr{X}_{(0, 1)} . 
	\end{multline}
 
Let $\nu_{1} \bigl(e, t(1,b) \bigr) \geq \cdots \geq \nu_{N} \bigl(e, t(1,b) \bigr) \geq 0$ be the square roots of the eigenvalues of $Dg_{dil}(x)^{T} \, Dg_{dil}(x)$. By \eqref{E:q.leq.d-p-1} and  \eqref{E:p.poz.a.big}, we have $q \leq d-p-1 < a$. Hence, by \eqref{E:asymptotics.of.eigvals.of.J} there exists $K' < \infty$ s.t.\ 
$\prod_{i=1}^{a} \nu_{i} \leq K' t^{-q} \leq K' t^{-(d-p-1)} = K' t^{p-d+1}$ for $t < 1$. (By \eqref{E:Sm'.countably.rect}, $a$ is an integer.)
Now, by \eqref{E:S3.defn}, $\mathscr{S}_{3} \subset \mathscr{X}$ and, by \eqref{E:S3.subset.Xtheta}, $x = \bigl( e, t(1,b) \bigr) \in \mathscr{S}_{3}$ 
implies $t \geq \theta$. Therefore we get 
$\prod_{i=1}^{a} \nu_{i} \leq K' t^{-q} \leq K' \theta^{p-d+1}$. Hence, by \eqref{E:hi.invrs.theta.m} implies that there exists $K_{8} < \infty$ s.t.\ 
    \begin{equation}  \label{E:prod.nu.i.leq.K.R^p-d+1}
      \prod_{i=1}^{a} \nu_{i} \leq K_{8} R^{p-d+1} \text{ for $R$ small}.
    \end{equation}
Actually, by \eqref{E:R<threeess}, $R$ is small. In fact, we are contemplating 
$R = R_{m} \to 0$. (See \eqref{E:Rm.=.dist.Sm.P}.)

Hence, by \eqref{E:integral.over.S3}, lemma \ref{L:rect.set.Jake.bounds}, and \eqref{E:prod.nu.i.leq.K.R^p-d+1},
	\begin{align*}  
		\Hm^{a} \bigl[ g_{dil}(\mathscr{S}_{3}) \bigr]  
		 &= \int_{\mathscr{S}_{3}} J^{\mathscr{S}_{3}} g_{dil} \bigl(e, t(1,b) \bigr) \; \
		   \Hm^{a} (de \, dt \, db) \\
		 &\leq \int_{\mathscr{S}_{3}} \prod_{i=1}^{a} \nu_{i} \bigl(e, t(1,b) \bigr) \; 
		   \Hm^{a} (de \, dt \, db) \\
		 & \leq K_{8} R_{m}^{p-d+1} \Hm^{a}( \mathscr{S}_{3} . 
	\end{align*}

Let us review the proof of lemma \ref{L:dilate.R.-(d-p-1)}. First, we reduce the problem to proving \eqref{E:f.dil.ineq.Ym}, where $\mcl{Y}_{m}$ is defined in \eqref{E:Ym.defn}. To simplify, we make a series of changes of variables, at first relying on \eqref{E:general.change.of.variables}. From the first change of variables we determine that to prove \eqref{E:f.dil.ineq.Ym} it suffices to prove \eqref{E:dilate.alpha.ineq}. After another change of variables we learn that in order to prove \eqref{E:dilate.alpha.ineq}, it suffices to prove \eqref{E:ineq.restricted.to.AxZ}. 
Another change of variables gets us from a hypothetical \eqref{E:dilate.ineq.in.C[U].with.h.invrs} to \eqref{E:not.quite.rite.alpha.f.alpha.invrs.ineq}. But the latter implies \eqref{E:ineq.restricted.to.AxZ}. Hence, it suffices to prove \eqref{E:dilate.ineq.in.C[U].with.h.invrs}. Changing variables again shows that it suffices to prove \eqref{E:dilate.ineq.in.H.invrs.AxZ}. 

So far our changes of variables have made use of \eqref{E:general.change.of.variables}. The last change of variables does not use of \eqref{E:general.change.of.variables}. We find that to prove \eqref{E:dilate.ineq.in.H.invrs.AxZ} we need ``only'' prove \eqref{E:Hm.a.gdil.S3.R-min.ineq}, which asserts that 
$\Hm^{a} \bigl[ g_{dil} (\mathscr{S}_{3}) \bigr] 
\leq K_{8} R^{p-d+1} \Hm^{a} (\mathscr{S}_{3})$.

In subsubsection \ref{SSS:bound.integral}, we express 
$K_{8} R^{p-d+1} \Hm^{a} (\mathscr{S}_{3})$ as an integral of something called 
``$J^{\mathscr{S}_{3}} g_{dil}$'' over $\mathscr{S}_{3}$. Lemma \ref{L:rect.set.Jake.bounds} shows us how to bound $J^{\mathscr{S}_{3}} g_{dil}$, and hence the integral, by a product of the largest eigenvalues of a certain matrix. The rest of the proof is an eigen analysis. The outcome is that \eqref{E:Hm.a.gdil.S3.R-min.ineq} holds.

This proves lemma \ref{L:dilate.R.-(d-p-1)}. Thus, we have an inequality of the form \eqref{E:Hm.a.S.geq.R.d-p-1}. But this $\gamma$ partly depends on the bi-Lipschita triangulation $f : |P| \to \D$ and choice of $\T \subset \Pf$. Take the supremum of such $\gamma$'s over all such triangulations and all appropriate $\T$'s to end up with a $\gamma$ depending only on $\D$, $C[\Pf]$, $a$, and $\F$.

Theorem \ref{T:lwr.bnd.on.Haus.meas} follows.. 

\section{Further remarks on measure}  \label{SS:remarks.on.measure}
 \begin{remark}[``local'' version of theorem] \label{R:local.level.theorem}
 The theorem gives a lower bound on $\Hm^{a}(\Ss')$ in terms 
of the $\Hm^{a}$-essential distance from $\Ss'$ to $\Pf$. Actually, what really matters 
is the $\Hm^{a}$-essential distance from $\Ss'$ to that portion of $\Pf$ that is 
\emph{``near''} $\T$. (Of course, this only matters if $\Pf \neq \T$.) Indeed, the proof of the theorem depends on this fact. See the beginning of subsection \ref{SS:R.geq.R0}.

In fact, it seems that one should be able to prove theorem \ref{T:lwr.bnd.on.Haus.meas} with the following refinement. 
Let $\clU$ be a neighborhood of $\T$ and let $R_{\clU}$ be the $\Hm^{a}$-essential distance from 
$\Ss'$ to $\Pf \cap \clU$. Then there exists a constant $\gamma_{\clU} > 0$, depending only 
on $\D$, $\T$, $C[\Pf]$, $a$, $\F$, \emph{and} $\clU$ s.t.\ 
$\Hm^{a}(\Ss') \geq \gamma_{\clU} R_{\clU}^{ \min(d-p-1,a) }$. Now, $R_{\clU}$ is no smaller than the $R$ in the theorem so this bound seems stronger than the one in the theorem. However, $\gamma_{\clU}$ also depends on $\clU$. Still, this ``local'' version should be stronger for small $R_{\clU}$.

However, as a practical matter, it is the global, not local, behavior that matters. 
First of all, $\Pf$ may be full of sets all homeomorphic to $\T$ and all of which can function as ``test pattern spaces''. (See section \ref{SS:D.T.plane.fit}.) 

More fundamentally, suppose $\Ss'$ lies far from $\T$ but close to some far flung corner 
of $\Pf$. Then one will have the worst of both worlds: $\Ss'$ can have large measure and one can still get the distasteful behavior described in section \ref{SS:measure.distance.to.P}.
 \end{remark}

\begin{remark}[Topologically based bounds on volume]  \label{R:topological.bounds.on.vol}
The proof of theorem \ref{T:lwr.bnd.on.Haus.meas} ultimately relies on the fact that the
$\Hm^{a}$-volume of underlying space of an $a$-dimensional sub complex, 
$O := \tilde{\Ss} \cap |Q|$, of the finite simplicial complex, $P$, is no smaller than the 
$\Hm^{a}$-volume of the smallest $a$-simplex in $P$. (See appendix \ref{Chptr:basics.of.simp.comps} and subsection \ref{SS:R.geq.R0}.) $O$ is derived from 
$\Ss'$ using theorem \ref{T:polyApproxThm}.

We found that the challenging case in the proof of the theorem is for integral $a$. Suppose 
$a$ is an integer and $(\Phi, \Ss', G, \T, a)$ satisfies the stronger property described in remark \ref{R:stronger.property}. (See also remark \ref{E:S.tilde.(co)homology}.) Then the argument that leads to \eqref{E:Hma(S.tilde)>0} leads to the stronger statement, $\check{H}^{a}(\tilde{\Ss}) \neq \{0\}$. Since outside $|Q|$ $\tilde{\Ss}$ has 0 $\Hm^{a}$ measure, then the $a$-dimensional \v{C}ech cohomology 
of $\tilde{\Ss}$ must be supported in $|Q|$ (okay, this requires proof), but 
$\tilde{\Ss} \cap |Q|$ is a subcomplex. Therefore, in \eqref{E:gamma.R0.defn} we can take $Vol_{a}(R_{0})$ to be the minimum $a$-dimensional volume of subcomplexes of $Q$ that have nontrivial \v{C}ech cohomology in dimension $a$ (which is the same as nontrivial simplicial cohomology in dimension $a$, Munkres \cite[Theorem 73.2, p.\ 437]{jrM84}.) This could considerably increase $\gamma$ in \eqref{E:Hm.a.S.geq.R.d-p-1}. 

(\eqref{E:diam.simps.<.R0/3} complicates this a bit. Maybe we need to make $R_{0}$ big enough so \eqref{E:diam.simps.<.R0/3} holds.)

Is Gromov \cite{mG82.VolumeAndCohom} relevant to this? 
\end{remark}  
 
  \begin{remark}[Computing $\gamma$]  \label{R:recipe.for.gamma}
Suppose we are given a compact Riemannian manifold $\D$ and 
$\Phi : \D \partlyto \F$ as usual. Conceptually, to compute a $\gamma$ valid in \eqref{E:Hm.a.S.geq.R.d-p-1} we proceed as follows. 
  \begin{enumerate}
\item Find a simplicial complex $P$ and triangulation $f : |P| \to \D$ as described in the statement of the theorem and compute $K_{f^{-1}}$ as in \eqref{E:triangulation.Lip.const}. 
\item Select a stratified space $\Pf \subset \D$ of perfect fits.   
\item Construct a cone bundle over $\Pf$ that fibers a neighborhood of $\Pf$ in 
$T \D \restriction_{\Pf}$ as in definition \ref{D:fibering.by.cones}.
\item Compute $R_{0}$ as in \eqref{E:R0.choice}.
\item Compute $Vol_{a}(R_{0})$ either as in \eqref{E:Vol.a(R0).defn} or as in remark \ref{R:topological.bounds.on.vol}. 
\item Compute $Q$ as in \eqref{E:Q.stays.away.from.T}.
\item Compute a value of $K(R_{0}) \geq 1$ as in \eqref{E:Ha.S.tilde.leq.mult.Ha.S'}. (Section 2.6.2 in \cite{spE09.PolyhedralApproxLong} develops a formula for that.)
\item Define $\gamma(R_{0})$ as in \eqref{E:gamma.R0.defn}. 
\item Compute a value of $K < \infty$ s.t. 
	\begin{equation}  \label{E:Ha.S'.S'dilate.ineq}
		\Hm^{a}(\Ss_{dilate}')
		   \leq K R^{-\min( d-p-1, a)} \Hm^{a}(\Ss') \text{ for } R \in (0, R_{0})  ,   
	\end{equation}
where $\Ss_{dilate}' := f_{dilate,R}(\Ss')$ and $f_{dilate}$ is defined in \eqref{E:when.f.dilate.is.identity} and \eqref{E:f.in.terms.of.F}. (Section \ref{SSS:bound.integral} illustrates one approach: Use the ``area formula'' 
(Hardt and Simon \cite[1.8 p.\ 13, p.\ 14, and p.\ 27]{rHlS86.GMT}, Federer \cite[3.2.3, p.\ 243 and 3.2.46, p.\ 282]{hF69} for $f_{dilate,R}$ in combination with lemma \ref{L:rect.set.Jake.bounds}.)  \label{I:fdil.K}
  \end{enumerate}

As in \eqref{E:Ha.and.S.dilate}, we have 
	\begin{equation*}  
		\Hm^{a}(\Ss_{dilate}') < \infty 
		  \text{ and } \Hm^{a}( \Ss_{dilate}' \cap \mcl{B}_{\ess} ) = 0.
	\end{equation*}
Therefore, as in \eqref{E:lwr.bnd.on.Ha.S.dilate} 
    \begin{equation*}
      \Hm^{a}(\Ss_{dilate}') \geq \gamma(R_{0}) R_{0}^{\min( d-p-1, a)} > 0 .
    \end{equation*}
Combining this with \eqref{E:Ha.S'.S'dilate.ineq} we get
    \begin{equation*}
      \Hm^{a}(\Ss') \geq K^{-1} \gamma(R_{0}) R_{0}^{\min( d-p-1, a)} 
        R^{\min(d-p-1, a)} \text{ for } R > 0 ,
    \end{equation*} 
as desired. 

A recipe better than this may well exist, but I expect computing a good general value for $\gamma$ would be quite difficult. A more realistic goal, I believe, is to compute one separately in each specific data map class of interest. Or even start with toy examples?
  \end{remark}

\begin{remark}  \label{R:Applications.of.min.meas.thm}
Here are a couple of areas where theorem \ref{T:lwr.bnd.on.Haus.meas} is helpful. 
In proposition \ref{P:aug.direct.mean.beats.robst.loc}, we use theorem \ref{T:lwr.bnd.on.Haus.meas} to show that the singular sets of measures of location on the circle that are ``resistant'' to outliers have ``asymptotically'' more measure than does that of a certain other specific kind of measure of location, even controlling the distance of the corresponding singular sets to the relevant $\Pf$'s. Another is the impact of the number of variables, in relation to the sample size, on the size of the singular set might be another (remark \ref{R:long.vs.wide}). 
\end{remark}

 \begin{remark}  \label{R:dist.moment.for.sing.sets}
   Theorem \ref{T:lwr.bnd.on.Haus.meas} suggests that the moments
      \[
         \int_{\Ss'} dist(x, \Pf)^{-s} \, \Hm^{a}(dx) ,
      \]
where $s = d-p-1$ or $d-p-2$ might be interesting.  
 \end{remark}

\chapter{Severity}  \label{Chptr:severity} 
Figure \ref{F:HeightFen} shows two data sets that appear to be quite close to singularities of $L^{1}$ or LAD linear regression (subsection \ref{SS:LAD}). In panel (b) of the figure the fitted difference in the lines before and after the perturbation in the data seems to be too small to be of practical importance. Of course, in that panel only one sort of perturbation of the data was tried, but assuming other perturbations in the data would lead to similar results (probably provably the case), this apparent singularity is not of grave concern. The issue here is that of \emph{severity} of singularity, the subject of this chapter. Not only is severity of practical interest, but we see how, by taking severity into consideration, some of the assumptions made in chapters \ref{Chptr:topology} and \ref{Chptr:Haus.meas.of.sing.set} can be relaxed (remark \ref{R:improvments.over.main.thm}). 

Let $\Phi : \D' \to \F$ be a data map and let $\Ss := \D \setminus \D'$ be its singular set. (We remind the reader or our default assumption \eqref{E:D'.=.D.less.S}.) The extent to which singularities make $\Phi$ unstable depends not only on how plentiful they are, but also on how severe they are. By definition, a singularity of $\Phi$ is a data set, $x_{0} \in \D$, in the vicinity of which a small change in the data, $x$, can cause a \emph{relatively} big change in $\Phi(x)$. \emph{Severity} of the singularity at $x_{0}$ (definition \ref{D:V-severe}) has to do with how big the change in $\Phi(x)$ can be in \emph{absolute} terms for a small change in $x$. More precisely, the severity of the singularity at $x_{0}$ has to do with the sizes of the images under $\Phi$ of arbitrarily small neighborhoods of $x_{0}$. (In this book all neighborhoods are open.) 

Sometimes a singularity at $x_{0}$ can be so severe that $\Phi$ maps any neighborhood of
$x_{0}$ \emph{onto} the feature space $\F$. In \cite{spE91.sings.plane.fit} we call such a singularity ``severe'' and show, essentially, (\cite[Theorem 2.5]{spE91.sings.plane.fit}) 
that in plane fitting (the subject of chapter \ref{Chptr:sings.in.plane.fit} below) one can reduce the size of the singular set of $\Phi$ by replacing nonsevere singularities by a smaller collection of severe ones. The definition of severe'' we use in this chapter is different from the one in 
\cite{spE91.sings.plane.fit}.

Severe singularities (in the \cite{spE91.sings.plane.fit} sense) seem to be optional, but it is typically impossible to avoid rather bad singularities. For example, if $\Phi$ is a measure of location for directional data (data on spheres; chapters \ref{Chptr:spherical.location}, \ref{Chptr:aug.direct.mean}, \ref{Chptr:robst.loc.on.circle}) then $\Phi$ must have many singularities,
$x_{0}$, s.t.\ there is no neighborhood $\clU$ of $x_{0}$ s.t.\ the closure of the image 
$\Phi ( \clU \cap \D')$ lies in any open hemisphere (\cite{spE91.top.direct.axis} 
and chapter \ref{Chptr:spherical.location}). Something similar happens for the location problem on projective spaces 
(Munkres \cite[p.\ 231]{jrM84}, \cite{spE91.top.direct.axis}).

Analogous to this in the plane fitting context is ``$90^{\circ}$ singularity'' (subsection \ref{SS:lin.combo.for.plane.fit.w/.k=nvar-1}). For simplicity, suppose $\Phi$ fits lines to bivariate data. We may assume that the fitted lines always pass through the origin. Then $\Phi$ has a $90^{\circ}$ singularity at $x_{0}$ if for no neighborhood, $\clU$, of $x_{0}$ does the closure of 
the image $\Phi(\clU \setminus \Ss)$ lie within the smaller angle made by any pair of nonperpendicular lines through the origin. The singularities shown in the (c) panels of figure \ref{F:LS.PC.LAD.lf.plots} are $90^{\circ}$ singularities of their respective plane (line) fitters.
We will see that a singularity, $x_{0}$, that is \emph{not} a $90^{\circ}$ singularity can be eliminated by modifying $\Phi$ in a neighborhood of $x_{0}$. 

In this chapter we generalize this idea 
(theorem \ref{T:if.lin.combo.on.F.then.can.rstrct.to.bad.sings}) to show that often one can replace a singular set $\Ss$ by a closed subset of $\Ss$ consisting of severe singularities. This is possible, e.g., if the feature space, $\F$, is a differentiable manifold and no symmetry properties are imposed on $\Phi$ (proposition \ref{P:smooth.manifs.have.commutative.convex.combos} and theorem \ref{T:if.lin.combo.on.F.then.can.rstrct.to.bad.sings}). 

  \begin{definition}  \label{D:V-severe}
      Let $\msf{V}$ be an open cover of $\F$. Let $\D' \subset \D$. 
      Let $\Phi : \D' \to \F$, let $\Ss \subset \D$, and suppose \eqref{E:D'.=.D.less.S} holds. Define 
         \begin{multline}  \label{E:S^V.notation}
            \Ss^{\msf{V}} = \Ss^{\msf{V}}(\Phi, \D') = \bigl\{  x \in \D : \text{if $\clU$ 
            is any neighborhood of $x$ }    \\
            \text{ then there does \emph{not} exist $V \in \msf{V}$ s.t.\ }
            \overline{ \Phi(\clU \cap \D') } \subset V \bigr\}.
         \end{multline}
      (Here, ``${}^{\overline{{\quad}}}$'' indicates closure.)  
      Call the points of $\Ss^{\msf{V}}$ ``$\msf{V}$-severe singularities of $\Phi$''. 
  \end{definition}

(This usage of ``severe'' is different from that appearing in \cite{spE91.sings.plane.fit}. See section \ref{SS:nondeterministic.data.maps} for a generalization.)  
Let $\msf{V_{1}}$ and $\msf{V_{2}}$ be covers of $\D$ and suppose that $\msf{V_{2}}$ is a refinement of $\msf{V}_{1}$. Recall that means that if $V_{2} \in \msf{V}_{2}$ then there exists $V_{1} \in \msf{V}_{1}$ s.t.\ $V_{2} \subset V_{1}$. 
Then clearly $\Ss^{\msf{V}_{1}} \subset \Ss^{\msf{V}_{2}}$. ``In the limit'', as it were, if $\msf{V}$ is a cover of $\F$ then $\Ss^{\msf{V}} \subset \Ss$, the ordinary singular set of $\Phi$. 

\emph{Claim:} 
   \begin{equation} \label{E:SV.is.closed}
      \Ss^{\msf{V}} \subset \Ss \text{ is \emph{closed}.}  
   \end{equation}
For suppose not. Then $\Ss^{\msf{V}}$ does not contain its boundary. Let $x \in \D \setminus \Ss^{\msf{V}}$ be a boundary point 
of $\Ss^{\msf{V}}$. Then $x$ has a neighborhood $\clU$ for which there exists $V \in \msf{V}$ s.t.\ $\overline{ \Phi(\clU \cap \D') } \subset V$. Since $x$ is a boundary point of $\Ss^{\msf{V}}$ there exists $y \in \Ss^{\msf{V}} \cap \clU$. But then $\clU$ is a neighborhood of $y$ and so $y \notin \Ss^{\msf{V}}$. Contradiction.

  \begin{remark}[Stronger notion of perfect fit space] \label{R:strong.perfect.fit}
The notion of severity suggests a strengthening of the concept of perfect fit space, $\Pf$ (section \ref{SS:calibration}). Let $\Pf$ be the perfect fit space of a class, $\blds{\Phi}$, of data maps, 
$\Phi$, with feature space $\F$. Suppose $x$ is a data set near $\Pf$. $x$ may be a singularity of a data map in $\blds{\Phi}$, but since it is almost a perfect fit, the set of plausible features in $\F$ describing $x$ should be small. This means that, for a given cover $\msf{V}$, $x$ should not be $\msf{V}$-severe. This suggests defining perfect fit space to be a pair $(\Pf, \clU)$, where $\clU$ is a neighborhood 
of $\Pf$, s.t.\ whenever $\Phi \in \blds{\Phi}$ then $\Phi$ behaves appropriately on $\Pf$ and has no severe singularities in $\clU$.

Connecting theorem \ref{T:lwr.bnd.on.Haus.meas} to this idea, it seems that if 
$(\Pf, \clU)$ is the perfect fit space for $\blds{\Phi}$, then the Hausdorff measure, of appropriate dimension, of severe singular sets of $\Phi \in \blds{\Phi}$ should be bounded below. Remark \ref{R:severity.trick} explains how sometimes this can be done. One of those times is described in corollary \ref{C:Haus.measure.of.severe.sings.on.spheres}.
  \end{remark}

  \begin{remark}[Measurability of $\Ss$] \label{R:S.measurable}
Let $\Phi \partlyto \F$ satisfy \eqref{E:D'.=.D.less.S} and let $\Ss$ be the singular set 
of $\Phi$. $\Ss$ may not be closed. (See remark \ref{R:Ss.closed?}.) By \eqref{E:D.metric.F.normal}, $\F$ is normal. Suppose $\F$ is actually a complete metric space. (See remark \ref{R:.completeness.of.F}.) Let $\msf{V}_{n}$ denote the cover of $\F$ consisting of all open balls of radius $1/n$ ($n = 1, 2, \ldots$). Let $x \in \Ss$. \emph{Claim:} $x$ is $\msf{V}_{n}$-severe for some $n$. For suppose not. Let $\{ x_{m} \} \subset \D'$ be a sequence converging to $x$. Then for every $n = 1, 2, \ldots$ there exists $m_{n}$ s.t.\ $\Phi(x_{m_{n}+1}), \Phi(x_{m_{n}+2}), \ldots$ all lie in a ball of radius $1/n$. Thus, $\bigl\{ \Phi(x_{m}) \bigr\}$ is a Cauchy sequence. Therefore, it converges. Since $x_{m} \to x$ is arbitrary, we have that $\Phi(x')$ converges as $x' \to x$ through $\D'$ contradicting the assumption that $x$ is a singularity. The claim is proved. Thus, $\Ss = \bigcup_{n} \Ss^{\msf{V}_{n}}$ and therefore, by \eqref{E:SV.is.closed}, when $\F$ is complete metric, 
$\Ss$ is Borel measurable. 
  \end{remark}

Let $\Phi$ be a data map with singular set $\Ss$. In this chapter we show that sometimes 
(often?)\ $\Phi$ can be replaced by a data map $\Omega$ all of whose singularities 
lie in $\Ss' := \Ss^{\msf{V}}$, the set of $\msf{V}$-severe singularities of the original data map 
$\Phi$. So $(\Omega, \Ss')$ satisifes the closure assumptions appearing  
in proposition \ref{P:sing.dim.when.H.d-r.D.=.0} and chapter 
\ref{Chptr:Haus.meas.of.sing.set} (in particular, property \ref{Pty:agree.near.T}). But not only do we get a closed set $\Ss'$ off which the data map is continuous, but $\Ss'$ consists of actual singularities, in fact, the most severe, hence, most interesting singularities of $\Phi$. Any lower bound on the dimension (or measure) of $\Ss^{\msf{V}}$ is automatically a lower bound on the dimension (resp., measure) of the original singular set, $\Ss$. 

In order to replace $\Ss$ by $\Ss^{\msf{V}}$ we eliminate the non-$\msf{V}$-severe singularities by smoothing them away. In order to do that we take local averages, i.e., ``convex combinations''. (See section \ref{SS:existence.of.convx.combos}.)
   \begin{definition}  \label{D:convex.combo.fn}
Let $\msf{V}$ be an open cover of $\F$. Let $k = 0, 1, 2, \ldots$ and let $\Delta_{k}$ 
be the $k$-dimensional simplex
   \begin{multline} \label{E:Delta.k.defn}
     \Delta_{k} := \Bigl\{ (\lambda_{0}, \lambda_{1}, \ldots, \lambda_{k}) \in \RR^{k+1} : 
       \lambda_{i} \geq 0 \, (i = 0, \ldots k)  \\
         \text{ and } \lambda_{0} + \lambda_{1} + \ldots + \lambda_{k} = 1 \Bigr\}.
   \end{multline}
If $V \in \msf{V}$, let $V^{k+1}$ denote the $(k+1)$-fold Cartesian power of $V$. 
We say that a function 
$\gamma_{k} : \bigcup_{V \in \msf{V}} \bigl( \{ V \} \times \Delta_{k} \times V^{k+1} \bigr) \to \F$ is a ``$k$-convex combination function'' on  $\msf{V}$ if: 
   \begin{enumerate}
      \item For every $V \in \msf{V}$,  
      $(\lambda_{0}, \lambda_{1}, \ldots, \lambda_{k}) \in \Delta_{k}$, and 
         $x_{0}, x_{1}, \ldots, x_{k} \in V$ we have \linebreak
         $\gamma_{k} \Bigl[ \bigl( V, (\lambda_{0}, \lambda_{1}, \ldots, \lambda_{k}), 
                  (x_{0}, x_{1}, \ldots, x_{k}) \bigr) \Bigr] \in V$.
         \label{I:convex.combos.are.local}
      \item More generally, $\gamma_{k}$ has the following consistency property:
         \begin{multline*}
            \text{If } V_{1}, V_{2} \in \msf{V}, \;
            (\lambda_{0}, \lambda_{1}, \ldots, \lambda_{k}) \in \Delta_{k},  
               \text{ and } x_{0}, x_{1}, \ldots, x_{k} \in V_{1} \cap V_{2} \text{ then }  \\
            \gamma_{k} \Bigl[ \bigl( V_{1}, (\lambda_{0}, \lambda_{1}, \ldots, \lambda_{k}), 
                    (x_{0}, x_{1}, \ldots, x_{k}) \bigr) \Bigr]  \\
            = \gamma_{k} \Bigl[ \bigl( V_{2}, (\lambda_{0}, \lambda_{1}, \ldots, \lambda_{k}), 
                 (x_{0}, x_{1}, \ldots, x_{k}) \bigr) \Bigr]
             \in V_{1} \cap V_{2}.
         \end{multline*}
      We may then regard $\gamma_{k}$ as a function on 
      $\bigcup_{V \in \msf{V}} \Delta_{k} \times V^{k+1}$.
                   \label{I:consistency.of.conv.combos}
      \item $\gamma_{k}$ is continuous on 
      $\bigcup_{V \in \msf{V}} (\Delta_{k} \times V^{k+1})$, which has its relative topology as a subset of $\Delta_{k} \times \F^{k+1}$. 
              \label{I:gamma.k.cont}
      \item If $(\lambda_{0}, \lambda_{1}, \ldots, \lambda_{k}) \in \Delta_{k}$ and 
      $x \in \F$ then
    \begin{equation}  \label{E:gamma.k.on.diagonal}
            \gamma_{k} \bigl[ (\lambda_{0}, \lambda_{1}, \ldots, \lambda_{k}), 
                 (x, x, \ldots, x)) \bigr] = x.
    \end{equation}
           \label{I:convex.combos.of.1.pt}
      \item $k=0$ case:   
      $\gamma_{0}(1, x) = x$, \; $x \in \F$. 
        \label{I:1.is.identity}
         \end{enumerate}
Suppose  
$\gamma :  \bigcup_{k=0}^{\infty} \bigcup_{V \in \msf{V}} \Delta_{k} \times V^{k+1} \to \F$
has the property that for every $k = 0, 1, \ldots$, the restriction of $\gamma$ to 
$\bigcup_{V \in \msf{V}} \Delta_{k} \times V^{k+1}$ is a $k$-convex combination function. Suppose further that ``terms with 0 coefficients'' can be dropped:
         \begin{multline}  \label{E:drop.0.coefs}
            \gamma \bigl[ (\lambda_{0},  
              \ldots, \lambda_{j-1}, 0, \lambda_{j+1}, \ldots,  \lambda_{k}),  
              (x_{0}, \ldots, x_{j-1}, x_{j}, x_{j+1}, \ldots, x_{k} )\bigr]  \\
                =  \gamma \bigl[ (\lambda_{0},  \ldots, \lambda_{j-1}, \lambda_{j+1}, 
                  \ldots,  \lambda_{k}),  
                    (x_{0}, \ldots, x_{j-1}, x_{j+1}, \ldots, x_{k} )\bigr],      \\
                     \text{ for every } (\lambda_{0}, \ldots, \lambda_{j-1}, \lambda_{j+1}, 
                       \ldots, \lambda_{k}) \in \Delta_{k-1}; \\
                        x_{0}, \ldots, x_{j-1}, x_{j}, x_{j+1}, \ldots, x_{k} \in V;  
                          \text{ and } V \in \msf{V}.
         \end{multline}
We then say that $\gamma$ is a ``convex combination function'' on $\msf{V}$. 
    \end{definition}
    
Explicit examples of convex combination functions are constructed in \cite{spE91.top.direct.axis} and in section \ref{SS:lin.combo.for.plane.fit.w/.k=nvar-1} and chapter \ref{Chptr:spherical.location} below. 

Suppose 
$\gamma :  \bigcup_{k=0}^{\infty} \bigcup_{V \in \msf{V}} \Delta_{k} \times V^{k+1} \to \F$
is a convex combination function on a cover $\msf{V}$ of $\D$. Suppose $\msf{V}'$ is an open cover of $\D$ that is a refinement of $\msf{V}$. The restriction of $\gamma$ to
$\bigcup_{k=0}^{\infty} \bigcup_{V \in \msf{V}'} \Delta_{k} \times V^{k+1}$, where this time the open sets $V$ are required to belong to $\msf{V}'$, may or may not be a 
convex combination function on the refinement $\msf{V}'$. If it is we say that $\gamma$ ``restricts'' to $\msf{V}'$.

(The expression ``convex combination mapping'' is used in Floater \cite{msF2002.PiecewiseLinearMaps}. However the phrase seems there to have a different meaning there than does our ``convex combination function''.) 

Suppose $x_{1}, \ldots, x_{j-1}, x_{j}, x_{j+1}, \ldots, x_{m} \in V \in \msf{V}$ but 
$x_{j} \in \F$ may or may not belong to $V$. 
We may then make \eqref{E:drop.0.coefs} a definition. Similarly, a, possibly, infinite convex combination
	 \begin{multline}  \label{E:infinite.convex.combo}
	   \gamma \bigl[ \bigl( ( \ldots, \lambda_{-j}, \lambda_{-j+1}, 
	         \ldots, \lambda_{j-1}, \lambda_{j}, \ldots),  
	                  ( \ldots, x_{-j}, x_{-j+1}, \ldots, x_{j-1}, x_{j}, \ldots) \bigr], \\
	                   \lambda_{j} \geq 0, \, x_{j} \in F, \, \sum_{j} \lambda_{j} = 1
	\end{multline} 
makes sense if only a finite set of $\lambda_{j}$'s are nonzero and the $x_{j}$'s corresponding to nonzero $\lambda_{j}$'s are all in the same $V \in \msf{V}$. Call the vector 
$( \ldots, \lambda_{-j}, \lambda_{-j+1}, \ldots, \lambda_{j-1}, \lambda_{j}, \ldots)$
the ``coefficients'' of $( \ldots, x_{-j}, x_{-j+1}, \ldots, x_{j-1}, x_{j}, \ldots)$

Call the expression \eqref{E:infinite.convex.combo} a ``$\gamma$-convex combination''. Say that $\gamma$ is ``commutative'' if the entries in the vector 
$( \ldots, \lambda_{-j}, \lambda_{-j+1}, \ldots, \lambda_{j-1}, \lambda_{j}, \ldots)$ can be permuted without changing the value of the convex combination providing that the same permutation is applied to \linebreak
$( \ldots, x_{-j}, x_{-j+1}, \ldots, x_{j-1}, x_{j}, \ldots)$. Commutativity will come up in theorem 
\ref{T:if.lin.combo.on.F.then.can.rstrct.to.bad.sings}. 

We may redundantly write the 0-convex function defined in part \ref{I:1.is.identity} of the preceding definition like this: $\gamma_{0}(V,1, x) = x$, \; $x \in V$, $V \in \msf{V}$. Written this way, 
$\gamma_{0}$ does itself have all the applicable properties, \emph{viz.}, \ref{I:convex.combos.are.local} through \ref{I:convex.combos.of.1.pt}.

Obviously one can define convex combinations satisfying properties \ref{I:convex.combos.are.local}, \ref{I:convex.combos.of.1.pt}, and \ref{I:1.is.identity} on any open set $V \subset \F$ s.t.\ $V$ is homeomorphic to a ball. But given a cover $\msf{V}$ of such sets, constructing convex combinations satisfying properties \ref{I:consistency.of.conv.combos} and \ref{I:gamma.k.cont} of the definition might be challenging. 

In chapter \ref{Chptr:spherical.location} it will be important that data maps satisfy a certain symmetry property. So in the following we assume that the data map is invariant under a, possibly trivial, group action. 

  \begin{remark}[Group invariance]  \label{R:symmetry}
Sometimes, as in \eqref{E:assume.Phi.S'.sym}, one wants a data map to be invariant under a group action. Let $\tilde{M}$ be a differentiable manifold and suppose a discrete group, $G$, acts ``freely'' and ``properly discontinuously'' on $\tilde{M}$ (Boothby \cite[pp.\ 94 -- 96]{wmB75}). Then by Boothby \cite[Theorem (8.3), p.\ 97]{wmB75} the  orbit space $\tilde{M}/G$ is also a differentiable manifold. This might allow one to reduce a problem with symmetry restrictions to one without such restrictions:  One just works on 
$\tilde{M}/G$. However, in chapter \ref{Chptr:spherical.location}, we consider a group that does not act freely, so the $\tilde{M}/G$ tactic might not work.

Let $\D$ be a metric space and let $G$ be a finite group of homeomorphisms of $\D$ onto itself. For $\mcl{A} \subset \D$, recall the definition, \eqref{E:G(x).G(A)}, of $G(\mcl{A}) := \bigl\{ g(x) \in \D : x \in \mcl{A}, \; g \in G \bigr\}$. Note that if $G(\mcl{A}) = \mcl{A}$ if and only if $g(\mcl{A}) = \mcl{A}$ for every $g \in G$. In this case we say that $\mcl{A}$ is ``$G$-invariant''. Assume $\D'$ is a $G$-invariant dense subset 
of $\D$. Let $\Ss = (\D')^{c}$. (``${}^{c}$'' denotes complementation relative to $\D$.) Then, by \eqref{E:g.commutes.w/.set.ops}, $\Ss$ is $G$-invariant. Let $\Phi: \D' \to \F$ be continuous. Assume that if $y \in \D'$ and  
$g \in G$ then $\Phi \bigl[g(y)\bigr] = \Phi(y)$. (``$\Phi$ is $G$-invariant'').  

Let $\msf{V}$ be an open cover of $\F$. \emph{Claim:} 
$\Ss^{\msf{V}}$ is $G$-invariant. Let $x \in \D$ and $g \in G$. 
Let $\mcl{W}$ be an open neighborhood of $g(x)$ and $V \in \msf{V}$. 
Since $\D'$ is $G$-invariant and $\Phi$ is $G$-invariant on $\D'$, by \eqref{E:g.commutes.w/.set.ops}, 
$\Phi \bigl[ g^{-1}(\mcl{W}) \cap \D' \bigr] = \Phi \bigl[ g^{-1}(\mcl{W} \cap \D' )\bigr]
= \Phi (\mcl{W} \cap \D')$. Thus, $\overline{\Phi \bigl[ g^{-1}(\mcl{W}) \cap \D' \bigr]} \subset V$ if and only if $\overline{\Phi (\mcl{W} \cap \D')} \subset V$. Since $g$ is a homeomorphism, any open neighborhood of $x$ is $g^{-1}(\mcl{W})$ for some open neighborhood $\mcl{W}$ of $g(x)$. Thus, $x \in \Ss^{\msf{V}}$ if and only if 
$g(x) \in \Ss^{\msf{V}}$. This proves the claim that $\Ss^{\msf{V}}$ is $G$-invariant.
  \end{remark}
 
Commutativity as a property of a convex combination function is defined one paragraph down from \eqref{E:infinite.convex.combo}. The following is the main result of this chapter. 
 
   \begin{theorem}  \label{T:if.lin.combo.on.F.then.can.rstrct.to.bad.sings} 
Let $\D$ be a locally compact, second countable metric space. (See \eqref{E:D.metric.F.normal}.) Let $G$ be a finite group of homeomorphisms of $\D$ onto itself. Let $m < \infty$ be the number of distinct elements of $G$. Let $\Ss \subset \D$ ($\Ss$ not necessarily closed) 
and let $\D' := \D \setminus \Ss$. Assume $\D'$ is dense in $\D$ and $G$-invariant. Thus, $\Ss$ is $G$-invariant. Let $\Phi: \D' \to \F$ be continuous and $G$-invariant.
Let $\msf{V}$ be an open cover of $\F$ so $\Ss^{\msf{V}}$ is $G$-invariant, by remark \ref{R:symmetry}. Suppose there is a convex combination function, $\gamma$, 
on $\msf{V}$ and $\gamma$ is commutative if $m > 1$. 
Then $\tilde{\D} := \D \setminus \Ss^{\msf{V}}$ 
satisfies $\D' \subset \tilde{\D}$ and $G(\tilde{\D}) = \tilde{\D}$. (See \eqref{E:G(x).G(A)}.) 
Let $\mathfrak{D}$ be an open covering of $\D$. \emph{Unless otherwise specified} 
$\mathfrak{D} = \{ \D \}$. 
 \begin{enumerate}
	\item There exists a continuous $G$-invariant data map 
	$\Omega := \Omega_{\gamma}: \tilde{\D} \to \F$ 
with the following property. If $x \in \tilde{\D}$ then there is a neighborhood 
$\clU \subset \D$ of $x$, $E \in \mathfrak{D}$, and $V \in \msf{V}$ s.t.\ 
$\clU \subset E$, $\overline{ \Phi(\clU \cap \D') } \subset V$, and $\Omega(x)$ 
is a $\gamma$-convex combination of points 
in $\Phi(\clU \cap \D')$, hence, by property \eqref{I:convex.combos.are.local} of definition \ref{D:convex.combo.fn}, 
$\Omega(x) \in V$. Thus, $\Omega$ approximates $\Phi$ on $\D'$ in the sense that if $x \in \D'$ then there exists 
$V  \in \msf{V}$ s.t.\ $\Phi(x) \in V$ and $\Omega(x) \in V$. \label{I:cont.approx.to.Phi.exists}
	\item Let $\Pf \subset \D$ be closed and $G$-invariant. Suppose  $\Pf \cap \D'$ is dense in $\Pf$. By \eqref{E:g.commutes.w/.set.ops}, $\Pf \cap \D'$ is $G$-invariant. Suppose the restriction 
	$\Phi \restriction_{\Pf \cap \D'}$ has a unique continuous extension, $\Theta$, 
	to $\Pf$. Let $\Rcl \subset \D$ be a neighborhood of $\Pf$ and suppose 
	$\rho : \Rcl \to \Pf$ is a retraction. Suppose $\Rcl$ is $G$-invariant and $\rho$ is $G$-equivariant, i.e., $\rho \circ g = g \circ \rho$ for every $g \in G$. If $\Omega$ is as in part \ref{I:cont.approx.to.Phi.exists}, then it can be modified in $\Rcl$ in such a way that it remains defined and continuous on $\tilde{\D}$, $G$-invariant, and agrees 
with $\Theta$ on $\Pf \cap \tilde{\D}$.  
	    \label{I:Omega.Phi.agree.on.Pf}
 \end{enumerate} 
   \end{theorem}

Since $\Omega$ is continuous on 
$\tilde{\D} := \D \setminus \Ss^{\msf{V}} \supset \D \setminus \Ss = \D'$, it follows that 
    \begin{equation}  \label{E:sing.set.of.Omega.in.severe.sings.of.Phi}
      \text{The singular set of $\Omega$ w.r.t.\  $\tilde{\D}$ 
        (and hence w.r.t.\ $\D'$) is a subset of } \Ss^{\msf{V}} .
    \end{equation} 

   \begin{remark}
For chapter \ref{Chptr:sings.in.plane.fit} we may take $G$ to be the trivial group (i.e., $m=1$) so for chapter \ref{Chptr:sings.in.plane.fit} the convex combination function $\gamma$ need not be commutative. See sections \ref{SS:convex.combos.by.recursion} and \ref{SS:lin.combo.for.plane.fit.w/.k=nvar-1}. But in chapters \ref{Chptr:spherical.location}, \ref{Chptr:aug.direct.mean}, and \ref{Chptr:robst.loc.on.circle}, $G$ will not be trivial so a commutative $\gamma$ will be needed. Such a $\gamma$ is constructed in section \ref{S:convex.combos.on.spheres}. 
  \end{remark}

Note that it might be possible to continuously extend $\Omega$ to points 
of $\Ss^{\msf{V}}$, as in lemma \ref{L:extend.Phi.to.D.less.S}. Afterwards, the singular set of $\Omega$ would be a proper subset of $\Ss^{\msf{V}}$. It is not clear when that maneuver would be useful

Note that if $\D$ is a finite dimensional manifold then, by Munkres \cite[pp.\ 3--4]{jrM66}, it is locally compact and second countable and, therefore, by Hocking and Young \cite[Corollary 2-59, p.\ 75]{jgHgsY61.Topology}, metrizable.
By \eqref{E:sing.set.of.Omega.in.severe.sings.of.Phi}, all the singularities of $\Omega$ w.r.t.\ $\tilde{\D}$ lie in $\Ss^{\msf{V}}$. The theorem is proved by eliminating non-severe singularities by using $\gamma$ to take local averages of values in $\F$. We call this operation ``smoothing''. 

In some cases, instead of using the $\Pf$ at hand, we may apply the the theorem with a compact test pattern space $\T \subset \Pf$, as in theorem \ref{T:Phi.star.Hr.contains.Theta.star.Hr} and proposition \ref{P:sing.dim.when.H.d-r.D.=.0}, playing the role of the $\Pf$ in part \ref{I:Omega.Phi.agree.on.Pf} of the theorem.
 
  \begin{remark}[Neighborhood Retracts] \label{R:retraction.in.manifs}
If $\D$ is a smooth manifold and $\Pf$ is an imbedded submanifold of $\D$ without boundary then, by Munkres \cite[Corollary 5.6, p.\ 53]{jrM66}, there exists an open neighborhood $\clU$ of $\Pf$ and a smooth retraction $r : \clU \to \Pf$. (Alternatively, one can appeal to Boothby \cite[Theorem (4.5), p.\ 193]{wmB75} and the Tubular Neighborhood Theorem, proposition \ref{P:tubular.nbhd.thm} above.) Examples will be given in \eqref{E:retraction.onto.Pk}, remark \ref{R:retraction.sphere.loc}, and in section \ref{SS:general.lin.class.thy}. 

However, in theorem \ref{T:if.lin.combo.on.F.then.can.rstrct.to.bad.sings}, $\Pf$ does not have to be a manifold. $\Pf$ might, e.g., be a stratified space (subsection \ref{SSS:conical.fibers}).   
In the notation of definition \ref{D:fibering.by.cones}, 
$Exp \circ \pi_{C} \circ \alpha : \mcl{C} \to \Pf$ is a retraction. See lemma \ref{L:Pk.is.nhbd.retract} for an example of a nontrivial $G$-equivariant retraction onto a stratified space. 
  \end{remark}   
   
  \begin{remark}[``Severity trick'']  \label{R:severity.trick}  
As discussed in remark \ref{R:Ss.closed?}, on the face of it proposition \ref{P:sing.dim.when.H.d-r.D.=.0} can only tell us about the dimension of a closed superset 
of $\Ss$, the singular set of $\Phi$. Similarly, theorem \ref{T:lwr.bnd.on.Haus.meas} only seems to give a lower bound on the Hausdorff measure of a closed superset of $\Ss$.

However, theorem \ref{T:if.lin.combo.on.F.then.can.rstrct.to.bad.sings} offers a possible way to bound $\dim \Ss$, the dimension of $\Ss$ itself and its measure, too. By \eqref{E:sing.set.of.Omega.in.severe.sings.of.Phi}, all singularities of $\Omega$ lie 
in $\Ss^{\msf{V}}$. If we can apply proposition \ref{P:sing.dim.when.H.d-r.D.=.0} 
to $\Omega$ with $\Ss' = \Ss^{\msf{V}}$, then we have a lower bound on 
$\dim \Ss^{\msf{V}}$, \emph{a fortiori} on $\dim \Ss$. Since severe singularities are more important than non-severe ones, bounding $\dim \Ss^{\msf{V}}$ below is the more important outcome of this strategy. It is in $\Phi$ that our interest lies. $\Omega$'s role is just to help us learn more about $\Phi$. We call ths maneuver the ``severity trick''. 

Similarly, with $\Pf := \T$, we might apply theorem \ref{T:Phi.star.Hr.contains.Theta.star.Hr} to $\Phi = \Omega$ and $\Ss' = \Ss^{\msf{V}}$ and conclude that $\Ss^{\msf{V}}$ is non-empty, a conclusion that then apply \emph{a fortiori} to $\Ss \supset \Ss^{\msf{V}}$. 

If $(\Phi, \Ss', G, \T, a)$ with $\Ss' := \Ss$ does not have property \ref{Pty:agree.near.T} the quintuplet $(\Omega, \Ss^{\msf{V}}, G, \T, a)$ might. In that case we may
apply theorem \ref{T:lwr.bnd.on.Haus.meas} with 
$\Omega$ in place of $\Phi$ and $\Ss^{\msf{V}}$ in place of $\Ss'$. Then, since $\Ss \supset \Ss^{\msf{V}}$ and, therefore, 
$dist^{a}(\Ss^{\msf{V}}, \Pf) \geq dist^{a}(\Ss, \Pf)$, inequality \eqref{E:Hm.a.S.geq.R.d-p-1} becomes
             \begin{equation*}  \label{E:Hm.a.SV.geq.R.d-p-1}
                \Hm^{a}(\Ss) \geq \Hm^{a}(\Ss^{\msf{V}}) 
                  \geq \gamma \; dist^{a}(\Ss^{\msf{V}}, \Pf)^{ \min(d-p-1,a) } 
                    \geq \gamma \; dist^{a}(\Ss, \Pf)^{ \min(d-p-1,a) }.
             \end{equation*}
To repeat, this holds whether or not $\Ss$ is closed. But even when \eqref{E:Hm.a.S.geq.R.d-p-1} holds for $\Ss' := \Ss$, 
by using $\Ss^{\msf{V}}$ we get a potentially bigger lower bound 
on $\Hm^{a}(\Ss)$ than we get with \eqref{E:Hm.a.S.geq.R.d-p-1}.
(\emph{Fine print:} To carry out the severity trick to use theorem \ref{T:lwr.bnd.on.Haus.meas}, we may have to replace $\Pf$ by a closed subset on which $\Phi$ has a continuous extension, as required in part \ref{I:Omega.Phi.agree.on.Pf} of theorem \ref{T:if.lin.combo.on.F.then.can.rstrct.to.bad.sings}.) 

Another important use of the severity trick is in checking \textbf{hypothesis \ref{Hyp:S.cap.T.small}} of theorem \ref{T:Phi.star.Hr.contains.Theta.star.Hr}. It is presumably easier to prove that there are few severe singularities in $\T$ (or, for application of theorem \ref{T:lwr.bnd.on.Haus.meas}, near $\Pf$) than to prove the same thing regarding singularities of any kind. See remark \ref{R:improvments.over.main.thm}. 

We will use this trick repeatedly in subsequent chapters.
  \end{remark}

  \begin{proof}[Proof of theorem \ref{T:if.lin.combo.on.F.then.can.rstrct.to.bad.sings} ]
By \eqref{E:D.metric.F.normal}, there exists a metric $\phi$ on $\D$. WLOG, 
$\phi$ is invariant under the action of $G$. I.e., $\phi$ is a metric w.r.t.\ which $G$ is a group of isometries. (This means, for every $g \in G$ and $x_{1}, \, x_{2} \in \D$ 
we have $\phi \bigl( g(x_{1}), g(x_{2}) \bigr) = \phi(x_{1}, x_{2})$.) If $\phi$ is not $G$-invariant, replace $\phi$ 
by $\bar{\phi} : (x_{1}, x_{2}) \mapsto m^{-1} \sum_{g \in G} \phi \bigl[ g(x_{1}), g(x_{2}) \bigr]$. If $x \in \D$ and $\delta > 0$ define the open ball $B_{\delta}(x) \subset \D$ w.r.t.\ the metric $\phi$. (See \eqref{E:ball.defn}.)

$\D'$ is $G$-invariant by assumption. By remark \ref{R:symmetry}, $\Ss^{\msf{V}}$ is also $G$-invariant and, therefore, so is $\tilde{\D}$. Since $\tilde{\D} = \D \setminus \Ss^{\msf{V}}$ we have that $\tilde{\D}$ is open, by \eqref{E:SV.is.closed}. Moreover, $\D' \subset \tilde{\D}$. We will construct a new data map $\Omega_{\gamma}$ continuous 
on $\tilde{\D}$. 

Because $\tilde{\D}$ is open, the space $\tilde{\D}$ is locally compact and second countable since $\D$ is. Then as in example \ref{Ex:nested.compacts}, there exists a sequence $\mcl{K}_{-2} = \mcl{K}_{-1} = \mcl{K}_{0} = \varnothing, \mcl{K}_{1}, \mcl{K}_{2}, \ldots$ of compact subsets of $\tilde{\D}$ whose union is $\tilde{\D}$ and satisfy $\mcl{K}_{i-1} \subset \mcl{K}_{i}^{\circ}$ ($i = 1, 2, \ldots$; $\mcl{K}_{i}^{\circ}$ is the interior of $\mcl{K}_{i}$). Replacing $\mcl{K}_{i}$ by $G(\mcl{K}_{i})$ (see \eqref{E:G(x).G(A)}) if necessary, we may assume $G(\mcl{K}_{i}) = \mcl{K}_{i}$ ($i = 1, 2, \ldots$): Let $i \in \ZZ$.
Then $G(\mcl{K}_{i})$ is compact, since $G$ is finite, and 
$G(\mcl{K}_{i}) \subset G(\mcl{K}_{i+1}^{\circ}) \subset G(\mcl{K}_{i+1})^{\circ}$.

Let $x \in \tilde{\D}$ and let $j = i > 0$ be the smallest $j$ s.t.\ $x \in \mcl{K}_{j}$. I.e., 
    \begin{equation*}
      x \in \mcl{K}_{i} \setminus \mcl{K}_{i-1} .
    \end{equation*}
Notice that $i$ is uniquely determined by $x$. Now, 
$\mcl{K}_{i+1}^{\circ} \setminus \mcl{K}_{i-1}$ is an open set containing $\mcl{K}_{i} \setminus \mcl{K}_{i-1}$. By definition of $\Ss^{\msf{V}}$, there exists $\delta(x) > 0$, 
$E \in \mathfrak{D}$, and $V(x) \in  \msf{V}$ s.t.\ 
   \begin{equation}   \label{E:bar.Phi.BD.in.VD}
     \overline{B_{\delta(x)}(x)} \subset E \cap (\mcl{K}_{i+1}^{\circ} 
       \setminus \mcl{K}_{i-1}) \cap \tilde{\D} \text{ and }
         \overline{ \Phi [B_{\delta(x)}(x) \cap \D'] }  \subset V(x).
   \end{equation}
Since $\mcl{K}_{i+1}$ is compact, we have,  
    \begin{equation} \label{E:B.delta.bar.is compact}
        \overline{ B_{\delta(x)}(x) } \text{ is compact.}
    \end{equation}

Let 
	\begin{equation}  \label{E:Ux.B.deltax/3}
	  \clU(x) := B_{\delta(x)/3}(x) 
	    \subset E \cap (\mcl{K}_{i+1}^{\circ} \setminus \mcl{K}_{i-1}) 
	      \cap \tilde{\D}.
	\end{equation}
In particular, by \eqref{E:bar.Phi.BD.in.VD}, 
	\begin{equation}  \label{E:Phi(x).in.Vx}
		\Phi(x') \in V(x), \text{ if } x \in \tilde{\D} \text{ and } x' \in  \clU(x) \cap \D'.
	\end{equation} 
Let $\mcl{E}(x) := B_{\delta(x)/6}(x) \subset \tilde{\D}$. Thus,
   \begin{equation}  \label{E:UD.in.K.i+2.less.K.i-1}
      x \in \overline{\mcl{E}(x)} \subset \clU(x) 
        \subset \mcl{K}_{i+1}^{\circ} \setminus \mcl{K}_{i-1}
   \end{equation}
and $\overline{\mcl{E}(x)}$ is compact (since it is closed and lies in $\mcl{K}_{i+1}$).
 
\emph{Claim:}  We may assume 
   \begin{equation}  \label{E:delta.radius.G.invar}
       \delta \bigl[ g(x) \bigr] = \delta(x),
   \end{equation}
First, since $G$ is a group of isometries and $\tilde{\D}$ and the $\mcl{K}_{j}$'s are $G$-invariant, for every $g \in G$ we have
   \[
      (\mcl{K}_{i+1}^{\circ} \setminus \mcl{K}_{i-1}) \cap \tilde{\D} 
        \supset g \bigl[ B_{\delta(x)}(x) \bigr] 
          = B_{\delta(x)} \bigl[ g(x) \bigr].
   \]
Similarly, 
$g \bigl[ \clU(x) \bigr] = g \bigl[ B_{\delta(x)/3}(x) \bigr] 
= B_{\delta(x)/3} \bigl[ g(x) \bigr]$.) 
By \eqref{E:g.commutes.w/.set.ops} and the fact that $\Phi$ and $\D'$ are $G$-invariant,
   \[
      \Phi \Bigl( B_{\delta(x)} \bigl[ g(x) \bigr] \cap \D' \Bigr) 
        = \Phi \Bigl( g \bigl[ B_{\delta(x)}(x) \cap \D' \bigr] \Bigr)
         = \Phi \bigl[ B_{\delta(x)} (x) \cap \D' \bigr].
   \]
In particular, 
   \[
      \overline{ \Phi \Bigl( B_{\delta(x)} \bigl[ g(x) \bigr] \cap \D' \Bigr) } \subset V(x).
   \]
Thus, we may take 
	\begin{equation*}
		\delta(x) = \min \Bigl\{ \delta \bigr[ g(x) \bigl] : g \in G \Bigr\}.
	\end{equation*} 
The claim \eqref{E:delta.radius.G.invar} follows. It also follows that
	\begin{equation}   \label{E:U.E.equivar}
		g \bigl[ \clU(x) \bigr] = \clU \bigl[ g(x) \bigr] \text{ and } 
		   g \bigl[ \mcl{E}(x) \bigr] = \mcl{E} \bigl[ g(x) \bigr].
	\end{equation}

Let $\nu = 1, 2, 3 \ldots$ and 
let $x \in \mcl{K}_{\nu} \setminus \mcl{K}_{\nu-1}^{\circ}$. 
If $x \in \mcl{K}_{\nu} \setminus \mcl{K}_{\nu-1} \subset \mcl{K}_{\nu} \setminus \mcl{K}_{\nu-1}^{\circ}$, then, by \eqref{E:UD.in.K.i+2.less.K.i-1}, we have 
$\overline{\mcl{E}(x)} \subset \clU(x) \subset \mcl{K}_{\nu+1}^{\circ} 
\setminus \mcl{K}_{\nu-1}$. The only other possibility is that $x$ belongs to the boundary 
$\mcl{K}_{\nu-1} \setminus \mcl{K}_{\nu-1}^{\circ} \subset \mcl{K}_{\nu-1} \setminus \mcl{K}_{\nu-2}$. In that case, by \eqref{E:UD.in.K.i+2.less.K.i-1} again, we have 
$\overline{\mcl{E}(x)} \subset \clU(x) \subset \mcl{K}_{\nu}^{\circ} 
\setminus \mcl{K}_{\nu-2}$. 
Thus, in general 
$\overline{\mcl{E}(x)} \subset \clU(x) \subset \mcl{K}_{\nu+1}^{\circ} 
\setminus \mcl{K}_{\nu-2}$.

By compactness, for every $\nu = 1, 2, \ldots$, there exists a finite collection 
$x_{\nu,  1}, \ldots, x_{\nu,  n_{\nu}} \in \mcl{K}_{\nu} \setminus \mcl{K}_{\nu-1}^{\circ} \subset \tilde{\D}$ s.t.\
    \begin{equation}  \label{E:E.x.nu.j.covers}
          \mcl{E}(x_{\nu,  j}) \subset \mcl{K}_{\nu+1}^{\circ} \setminus \mcl{K}_{\nu-2} 
            \; (j = 1, \ldots, n_{\nu}) \text{ cover }
              \mcl{K}_{\nu} \setminus \mcl{K}_{\nu-1}^{\circ}. 
    \end{equation}
\emph{A fortiori}, by \eqref{E:UD.in.K.i+2.less.K.i-1} again and invariance 
of the $\mcl{K}_{i}$'s, we have
	\begin{equation}  \label{E:G.U.subset.K.diff}
		G \bigl[ \clU(x_{\nu,  j}) \bigr] 
		  \subset \mcl{K}_{\nu+1}^{\circ} \setminus \mcl{K}_{\nu-2}, 
	            \quad j = 1, \ldots, n_{\nu},
	\end{equation}
cover $\mcl{K}_{\nu} \setminus \mcl{K}_{\nu-1}^{\circ}$. 

Observe that 
    \begin{equation}  \label{E:finite.Ui's.intrsct.Kell}
        \text{for each } \ell = 1,2, \ldots, \text{ only finitely many sets } 
          G \bigl[ \clU(x_{\nu,  j}) \bigr] \text{ intersect } \mcl{K}_{\ell}.
    \end{equation} 
To see this, let $\ell = 1, 2, \ldots$ and suppose for some $\nu$ and $j$ we have 
$G \bigl[ \clU(x_{\nu,  j}) \bigr] \cap \mcl{K}_{\ell} \neq \varnothing$. Now, by \eqref{E:G.U.subset.K.diff} and the fact that $\mcl{K}_{i} \uparrow$, we must have 
$\nu -2 < \ell$. Therefore, the number of sets $G \bigl[ \clU(x_{\nu,  j}) \bigr]$ intersecting 
$\mcl{K}_{\ell}$ is at most $n_{1} + \cdots + n_{\ell+1}$. In particular, at most $n_{1} + \cdots + n_{\ell+1}$ of the $x_{\nu, ,j}$'s lie in $\mcl{K}_{\ell}^{\circ}$. Thus, the collection $\{ x_{\nu, ,j} \in \tilde{\D} :\nu = 1, 2, \ldots ; j = 1, \ldots, n_{\nu} \}$ is locally finite. 

Recall that $m$ is the number of elements of $G$. Relabel the points $x_{\nu,  j}$, 
($j = 1, \ldots, n_{\nu}; \nu = 1, 2, \ldots$) as $x_{km + 1}$, $k = 0, 1, 2, \ldots$. (Just to be clear, here ``$km + 1$'' is not a double index. 
It means ``$k$ times $m$ plus 1''.) Thus, there are gaps of length $m-1$ in the indexing. 
By \eqref{E:E.x.nu.j.covers}, \eqref{E:UD.in.K.i+2.less.K.i-1}, \eqref{E:G.U.subset.K.diff}, and \eqref{E:finite.Ui's.intrsct.Kell}, $\{ G \bigl[ \clU(x_{jm+1}) \bigr], \; j = 0, 1, \ldots \}$ is a locally finite open covering of $\tilde{\D}$. 

Let 
    \begin{equation}  \label{E:X.jm+1.in.U(x.jm+1).cap.D'}
      X_{jm+1} \in  \clU(x_{jm+1}) \cap \D' 
    \end{equation}
be arbitrary ($j = 0, 1, 2, \ldots$ ). 
($X_{jm+1}$ exists since $\D'$ is dense in $\D$. 
Thus, $\Phi(X_{jm+1})$ is defined for every $j$.) Again we have gaps of length $m-1$. We fill them in. Let $g_{1}, \ldots, g_{m}$ be the elements of $G$ with $g_{1}$ being the identity element. Recall \eqref{E:U.E.equivar}. Let 
   \begin{multline}   \label{E:G.invariance.of.X's}
      X_{jm+k} := g_{k}(X_{jm+1}), \; x_{jm+k} := g_{k}(x_{jm+1}), \;
       \mcl{E}_{jm+k} := \mcl{E}(x_{jm+k}) := g_{k}(\mcl{E}_{jm+1}), \\
         \text{ and } \clU_{jm+k} := \clU(x_{jm+k}) := g_{k}( \clU_{jm+1})  \quad
      (j = 0, 1, 2, \ldots; k = 1, \ldots, m). 
   \end{multline}
Thus, since $\D'$ and $\tilde{\D}$ are $G$-invariant, we have that 
    \begin{equation}  \label{E:X.i.in.U(x.i).cap.D'}
      X_{i} \in \clU(x_{i}) \cap \D' \text{ and } x _{i} \in \tilde{D} \text{ for every } i = 1, 2, \ldots.
    \end{equation}
 In addition, $X_{i} \in \clU_{i}$ ($i=1, 2, \ldots$) and for every $g \in G$ and $i =  1, 2, \ldots$ there exists $j = 1, 2, \ldots$ s.t.\ $g(\clU_{i}) = \clU_{j}$. Moreover, $\clU_{i}, i=1, 2, \ldots$ is a locally finite cover of $\tilde{\D}$. Since, by \eqref{E:E.x.nu.j.covers}, $G \bigl[ \mcl{E}(x_{km+1}) \bigr]$ ($k = 1, 2, \ldots$) covers $\tilde{\D}$, it follows that 
	\begin{equation}   \label{E:Ei.covers.D.tilde}
		\mcl{E}_{i}, \; i = 1, 2, \ldots, \text{ cover } \tilde{\D}.
	\end{equation} 

The proof of Rudin \cite[Theorem 2.13, p.\ 40]{wR66.realcmplx} 
easily extends to show that there exists a (continuous) partition of unity 
(Munkres \cite[Definition 2.5, p.\ 20]{jrM66}) on $\tilde{\D}$ s.t.\ for every $i$ we have 
$supp \, f_{i} \subset \clU_{i}$.\footnote{(This footnote might be continued on the next page.) In fact, by Urysohn's lemma (Simmons \cite[Theorem A, p.\ 135]{gfS63}; by \eqref{E:D.metric.F.normal}, $\F$ is normal) and \eqref{E:UD.in.K.i+2.less.K.i-1}, for each $i$ there exists a continuous function $\gamma_{i} : \tilde{\D} \to [0,1]$ s.t.\ 
	\begin{equation*}
		\gamma_{i}(x) =
			\begin{cases}
				1, &\text{ if } x \in \overline{\mcl{E}_{i}}, \\
				0, &\text{ if } x \in \tilde{\D} \setminus \clU_{i}.
			\end{cases}
	\end{equation*}
Let $f_{1} := \gamma_{1}$ and for $j > 1$, let $f_{j} := (1-\gamma_{1}) (1-\gamma_{2}) \cdots (1-\gamma_{j-1}) \gamma_{j}$. 
So $supp \, f_{j} \subset \clU_{j}$. Then, by induction, for every $n = 1, 2, \ldots$, 
	\[
		f_{1} + \cdots + f_{n} = 1 - (1-\gamma_{1}) (1-\gamma_{2}) \cdots 
		  (1-\gamma_{n}).
	\]
Let $x \in \tilde{\D}$. By \eqref{E:Ei.covers.D.tilde}, there exists $i$ s.t.\ $x \in \overline{\mcl{E}_{i}}$. Therefore, $\gamma_{i}(x) = 1$, so $j > i$ implies $f_{j}(x) = 0$ 
and $\sum_{j=1}^{n} f_{j}(x) = 1$ for every $n \geq i$. 
I.e., $\sum_{j=1}^{\infty} f_{j}(x) = 1$.} 
(Here, ``$supp$'' means ``support.'' Conceivably, $f_{i} \equiv 0$ for some $i$'s.)

We employ an idea used in the proof of \eqref{E:smooth.<.min}. For $j = 0, 1, 2, \ldots$ 
and $k = 1, \ldots, m$ let
   \[
      \bar{f}_{jm+k} = m^{-1} \sum_{\ell=1}^{m} f_{jm+\ell} \circ g_{\ell} \circ g_{k}^{-1}.
   \]  
We introduce the $\bar{f}_{jm+k}$'s in order to get:
   \begin{equation}    \label{E:G.equivariance.of.f.bar}
      \text{If } g_{k'} = g^{-1} \circ g_{k} \text{ then } \bar{f}_{jm+k} \circ g = \bar{f}_{jm+k'} 
         \quad (g \in G; \; j = 0, 1, 2, \ldots; \;k = 1, \ldots, m).
   \end{equation}
Note that for any function $\mathscr{F} : \D \to \RR$ we have
    \begin{equation}  \label{E:g.and.supp.fi}
       g(supp \, \mathscr{F}) = supp \, (\mathscr{F} \circ g^{-1}) .
    \end{equation}
    
We \emph{claim} that $\{ \bar{f}_{i} \}$ is a partition of unity on $\tilde{\D}$ s.t.\ for every $i$ we have $supp \, \bar{f}_{i} \subset \clU_{i}$.
Let $j = 0, 1, 2, \ldots$ and $k = 1, \ldots, m$. Observe that, since $f_{i} \geq 0$ for every $i$,
   \[
      supp \, \bar{f}_{jm+k} = \bigcup_{\ell} supp \, (f_{jm+\ell} \circ g_{\ell} \circ g_{k}^{-1} )
   \]
so, by \eqref{E:g.and.supp.fi} and \eqref{E:G.invariance.of.X's},
   \[
      supp \, (f_{jm+\ell} \circ g_{\ell} \circ g_{k}^{-1} ) 
        = g_{k} \circ g_{\ell}^{-1} ( supp \, f_{jm+\ell} )
         \subset g_{k} \circ g_{\ell}^{-1} \bigl( \clU_{jm+\ell} \bigr) 
           = \clU_{jm+k}.
   \]
I.e., 
	\begin{equation}  \label{E:supp.f.bar.subset.U}
		supp \, \bar{f}_{jm+k} \subset \clU_{jm+k}. 
	\end{equation} 
In particular, $\{ supp \, \bar{f}_{i}, \, i = 1, 2, \ldots \}$ is a locally 
finite collection of subsets of $\tilde{D}$.

To complete proof of the claim we must show that
   \begin{equation}   \label{E:bar.f.is.part.of.1}
      \sum_{i \geq 1} \bar{f}_{i}(x) = 1 \text{ for every } x \in \D.
   \end{equation}
Since $\{ f_{i} \}$ is a partition of unity on $\tilde{\D}$, we have, on $\tilde{\D}$, 
   \begin{align*}
   \sum_{i \geq 1} \bar{f}_{i} &= m^{-1} \sum_{j\geq 0} 
          \sum_{k=1}^{m} \sum_{\ell=1}^{m} f_{jm+\ell} \circ g_{\ell} \circ g_{k}^{-1}  \\
     &= m^{-1} \sum_{g \in G} \sum_{j\geq 0} \sum_{\ell=1}^{m} 
           f_{jm+\ell} \circ g_{\ell} \circ (g^{-1} \circ g_{\ell})^{-1}  \\
     &= m^{-1} \sum_{g \in G} \sum_{j\geq 0} \sum_{\ell=1}^{m} f_{jm+\ell} \circ g  \\
     &= m^{-1} \sum_{g \in G} 1  \\
     &= 1.
   \end{align*}
This proves \eqref{E:bar.f.is.part.of.1} and the claim that $\{ \bar{f}_{i} \}$ is a partition of unity on $\tilde{\D}$ subordinate to $\{ \clU_{i} \}$.

Recall that $\gamma$ is a convex combination function on $\msf{V}$, commutative 
if $m > 1$. \emph{Claim:}   
   \begin{multline}  \label{E:Xi.is.defined.and.cont}
      \text{The formula } \Xi(x) = \gamma \Bigl[ \bigl( \bar{f}_{i}(x), i = 1, 2, \ldots \bigr), 
         \bigl( \Phi(X_{i}), i = 1, 2, \ldots \bigr) \Bigr], \quad x \in \tilde{D} \\
           \text{ defines a continuous function } \tilde{\D} \to \F .
   \end{multline}
Let $x \in \tilde{\D}$. By local finiteness of $\clU_{i}$, 
$i = 1, 2, \ldots$, there are only finitely many $\clU_{i}$'s s.t.\ $x \in \clU_{i}$. Only if $x \in \clU_{i}$ might we have $\bar{f}_{i}(x) > 0$ so only finitely many of the ``coefficients'' 
$\bar{f}_{i}(x)$ are nonzero. By \eqref{E:infinite.convex.combo}, to prove that $\Xi(x)$ is defined we must show that for some $V \in \msf{V}$ we have 
$\bigl\{ \Phi(X_{i}) \in \F : \bar{f}_{i}(x) > 0 \bigr\} \subset V$. We then will still have to establish continuity. 

Let 
	\begin{equation}  \label{E:Y=.intersect.Ui}
		\Y(x) := \bigcap_{x \in \clU_{i}} \clU_{i}, \qquad x \in \tilde{D} .
	\end{equation}
Since $x$ only belongs to a non-empty finite collection of $\clU_{i}$'s, we have 
that $\Y(x)$ is open. \emph{Claim:} $\Y(x)$ is ``equivariant'' in $x$. I.e., 
	\begin{equation} \label{E:Y.commutes.w/.G}
		g \bigl[ \Y(x) \bigr] = \Y \bigl[ g(x) \bigr] \text{ for ever } x \in \tilde{D} 
		  \text{ and } g \in G.
	\end{equation}
Let $g \in G$. Now, for $i = 1, 2, \ldots$, we have $x \in \clU_{i}$ if and only if 
$g(x) \in g(\clU_{i})$ and, by \eqref{E:G.invariance.of.X's}, $g(\clU_{i}) = \clU_{j}$ 
for some $j = 1, 2, \ldots$. Thus, 
$\bigl\{ g(\clU_{i}) : x \in \clU_{i} \bigr\} \subset \bigl\{ \clU_{j} : g(x) \in \clU_{j} \bigr\}$. Therefore, by \eqref{E:g.commutes.w/.set.ops}, 
    \begin{equation} \label{E:g.Y.x.subset.Y.g.x}
        g \bigl[ \Y(x) \bigr] = \bigcap_{x \in \clU_{i}} g(\clU_{i}) 
          \subset \bigcap_{g(x) \in \clU_{j}} \clU_{j} = \Y \bigl[ g(x) \bigr] .
    \end{equation}
By replacing $g$ by $g^{-1}$ and $x$ by $g(x)$ in the preceding we get 
$g^{-1} \Bigl( \Y \bigl[ g(x) \bigr] \Bigr) \subset \Y(x)$. 
I.e., $\Y \bigl[ g(x) \bigr]  \subset g \bigl[ \Y(x) \bigr]$. Combining this with \eqref{E:g.Y.x.subset.Y.g.x}, we get \eqref{E:Y.commutes.w/.G}, as claimed.
 
By \eqref{E:supp.f.bar.subset.U}, for every $i = 1, 2, \ldots$, 
$supp \, \bar{f}_{i} \subset \clU_{i}$. Let $\overline{\Y(x)}$ be the closure of $\Y(x)$ in $\D$. By \eqref{E:finite.Ui's.intrsct.Kell}, only finitely many $\{ \clU_{i} \}$'s intersect any given 
$\mcl{K}_{\ell}$. It follows from \eqref{E:UD.in.K.i+2.less.K.i-1} that 
$\overline{\Y(x)} \subset \mcl{K}_{\ell}$ for some $\ell$.  Hence, by \eqref{E:supp.f.bar.subset.U}, there are only finitely 
many $\bar{f}_{i}$'s, say, $\bar{f}_{i_{1}} \ldots , \bar{f}_{i_{N_{2}}}$, whose support intersects $\overline{\Y(x)}$. 
Let $x \in \tilde{\D}$ so $x \in \Y := \Y(x)$. 
    \begin{equation}  \label{E:f.i1,...,f.iN2}
      \text{Let } \bar{f}_{i_{1}} \ldots , \bar{f}_{i_{N_{2}}} 
        \text{ be the only $\bar{f}$'s whose support intersects } \overline{\Y} .
    \end{equation} 
Relabeling if necessary, there exists $N_{1} = 1, 2, \ldots, N_{2}$ s.t.\  
	\begin{multline}  \label{E:x.in.bar.f.ij.in.Uij}
		\text{for } j = 1, \ldots, N_{1} \text{ we have }
		x \in supp \, \bar{f}_{i_{j}} \subset \clU_{i_{j}} \\
		    \text{  but if } j = N_{1} + 1, \ldots, N_{2} 
		      \text{ then } x \notin supp \, \bar{f}_{i_{j}}.
	\end{multline} 
Recall that $\phi$ is the metric on $\D$. Now, by \eqref{E:G.invariance.of.X's} and \eqref{E:Ux.B.deltax/3},  
    \begin{equation} \label{E:x.within.delta.of.x.ij}
      \phi(x', x_{i_{j}}) < \delta(x_{i_{j}})/3 \text{ if } x' \in \clU_{i_{j}} = \clU(x_{i_{j}}) .
    \end{equation}

Let 
	\begin{equation}   \label{E:Y.tilde(x).defn}
	     \tilde{\Y} := \tilde{\Y}(x) := \Y(x) \setminus 
	       \left( \bigcup_{j = N_{1}+1}^{N_{2}} supp \, \bar{f}_{i_{j}} \right)
	           \subset \clU_{i_{1}} 
	\end{equation}
(Yes, $i_{1}$.) Thus, $\tilde{\Y}$ is open and $x \in \tilde{\Y}$. 
Thus, if $x' \in \tilde{\Y}$ then, by \eqref{E:drop.0.coefs} and \eqref{E:infinite.convex.combo}, 
    \begin{multline} \label{E:sum.1.to.N2.=1}
            \ell \notin \{ i_{1}, \ldots, i_{N_{1}} \} \text{ implies } 
                 \bar{f}_{\ell}(x') = 0, \text{ so } 
                   \bar{f}_{i_{1}}(x') + \cdots + \bar{f}_{i_{N_{1}}}(x') = 1, \;
                \text{ and } \\
               \Xi(x') = \gamma \Bigl[ \bigl( \bar{f}_{i_{j}}(x'), j = 1, 2, \ldots, N_{1} \bigr), 
                \bigl( \Phi(X_{i_{j}}), j = 1, 2, \ldots, N_{1} \bigr) \Bigr],
    \end{multline}
providing we show that for $V \in \msf{V}$ we have 
$\Phi(X_{i_{j}}) \in V$ ($j = 1, 2, \ldots, N_{1}$). Now, by \eqref{E:X.i.in.U(x.i).cap.D'}, 
    \begin{equation*}
      X_{i_{j}} \in \clU(x_{i_{j}}), \quad j = 1, 2, \ldots, N_{1} .
    \end{equation*}

By \eqref{E:G.invariance.of.X's} and \eqref{E:Ux.B.deltax/3}, 
$\clU_{i_{j}} =  B_{\delta(x_{i_{j}})/3}(x_{i_{j}})$ ($j = 1, \ldots, N_{1}$).
We may assume, 
    	\begin{equation}  \label{E:delta.i1.biggest}
    		\delta (x_{i_{1}}) \geq \cdots  \geq \delta (x_{i_{N_{1}}}) > 0.
    	\end{equation}
 
By \eqref{E:X.jm+1.in.U(x.jm+1).cap.D'} and \eqref{E:G.invariance.of.X's}, we have 
$X_{i_{j}} \in \clU_{i_{j}}$ ($j = 1, 2, \ldots, N_{1}$). 
And by \eqref{E:x.within.delta.of.x.ij} and \eqref{E:x.in.bar.f.ij.in.Uij}, we have 
$x \in B_{\delta(x_{i_{1}})/3}(x_{i_{1}})$ and  
   \begin{multline*}
      \phi \bigl( X_{i_{j}}, x_{i_{1}} \bigr) 
        \leq \phi \bigl( X_{i_{j}}, x_{i_{j}} \bigr) 
          + \phi \bigl( x_{i_{j}}, x \bigr) + \phi \bigl( x, x_{i_{1}} \bigr)    \\
            <  \tfrac{1}{3} \delta \bigl( x_{i_{j}} \bigr) 
              +  \tfrac{1}{3} \delta \bigl( x_{i_{j}} \bigr) 
                +  \tfrac{1}{3} \delta \bigl( x_{i_{1}} \bigr) 
                  \leq  \delta \bigl( x_{i_{1}} \bigr),  \quad j = 1, \ldots, N_{1}.
   \end{multline*}
I.e., 
	\begin{equation}  \label{E:Xis.x.in.B.delta.x1}
		X_{i_{1}}, \ldots, X_{i_{N_{1}}} \text{  and } x \text{ all lie in } 
		           B_{\delta(x_{i_{1}})}(x_{i_{1}}). 
	\end{equation}

But by \eqref{E:bar.Phi.BD.in.VD}, 
    \begin{equation}   \label{E:Phi(B.delta.i1).subset.V}
      \Phi(\clU_{i_{1}}) \subset 
        \overline{ \Phi \Bigl( B_{\delta(x_{i_{1}})} (x_{i_{1}}) \cap \D' \Bigr) } 
          \subset V(x_{i_{1}}) \in \msf{V} .
    \end{equation}
Now, $X_{i_{j}} \in \D'$, so, by 
\eqref{E:Xis.x.in.B.delta.x1} and \eqref{E:Phi(B.delta.i1).subset.V}, 
$\Phi(X_{i_{j}}) \in V(x_{i_{1}})$ ($j = 1, \ldots N_{1}$). Therefore, convex combinations 
of $\Phi(X_{i_{1}}), \ldots, \Phi(X_{i_{N_{1}}})$ exist and, hence by \eqref{E:sum.1.to.N2.=1}, $\Xi(x)$ is defined. 

Moreover, writing $\clU := B_{\delta(x_{i_{1}})}(x_{i_{1}})$ and $V := V(x_{i_{1}})$, by \eqref{E:x.in.bar.f.ij.in.Uij}, we have, in the language of part \ref{I:cont.approx.to.Phi.exists} of  the theorem, that $\clU$ is a neighborhood of $x$ and $\Xi(x)$ is a $\gamma$-convex combination of points in $\Phi(\clU \cap \D')$. In addition, by \eqref{E:x.in.bar.f.ij.in.Uij}
and \eqref{E:Phi(B.delta.i1).subset.V}, if $x \in \D'$ then $\Phi(x) \in V$ as well. 
I.e., $\Xi$ has the ``approximation'' property promised in the theorem. Furthermore, by \eqref{E:bar.Phi.BD.in.VD}, there exists $E \in \mathfrak{D}$ s.t.\ $\clU \subset E$. 

Suppose $x' \in \tilde{\Y}$. Then, by \eqref{E:Y.tilde(x).defn} and \eqref{E:Y=.intersect.Ui}, we have $x' \in \Y = \Y(x) \subset \clU_{i_{1}}$. There exists $i$ s.t.\ $\bar{f}_{i}(x') > 0$. In that case, $(supp \, \bar{f}_{i}) \cap \overline{\Y}$ (a bar over $\Y$) contains 
$(supp \, \bar{f}_{i}) \cap \tilde{\Y} \ne \varnothing$ (a tilde over $\Y$). By \eqref{E:f.i1,...,f.iN2}, this implies $i = i_{j}$ for some $j = 1, \ldots, N_{2}$, in fact $j \leq N_{1}$. Hence, $\Xi(x') = \gamma \Bigl[ \bigl( \bar{f}_{i_{j}}(x'), j = 1, 2, \ldots, N_{1} \bigr), 
         \bigl( \Phi(X_{i_{j}}), j = 1, 2, \ldots, N_{1} \bigr) \Bigr]$ is defined with coefficients continuous in $x' \in \tilde{\Y}$, so, by part \ref{I:gamma.k.cont} of definition \ref{D:convex.combo.fn}, $\Xi(x')$ is continuous in $x' \in \tilde{\Y}$. 
Since $x \in \tilde{\D}$ is arbitrary, the claim \eqref{E:Xi.is.defined.and.cont} that $\Xi$ is defined and continuous on $\tilde{\D}$ is proved. 

Therefore, all singularities of $\Xi$ lie in $\Ss^{\msf{V}}$. Moreover, if 
$x' \in \tilde{\Y} \cap \D'$ then, by virtue of the fact that $x' \in \clU_{i_{1}}$, 
we have $\Phi(x') \in V(x_{i_{1}})$, by \eqref{E:bar.Phi.BD.in.VD}. Thus, by \eqref{E:Xis.x.in.B.delta.x1}, \eqref{E:Y.tilde(x).defn}, and \eqref{E:Phi(B.delta.i1).subset.V}, 
    \begin{equation}   \label{E:Phi.Xi.x'.in.same.V}
        \text{if } x \in \tilde{\D} \text{ then there exists } V \in \msf{V} \text{ s.t.\ } \\
          \Xi \bigl[ \tilde{\Y}(x) \cap \tilde{\D} \bigr] 
            \cup \overline{ \Phi \bigl[ \tilde{\Y}(x) \cap \D' \bigr] } \subset V.
    \end{equation} 
    
For future reference we prove the \emph{claim:}
    \begin{equation}  \label{E:Y.tilde.G.equivar}
        \tilde{\Y} \text{ is } G\text{-equivariant.}
    \end{equation} 
i.e., for $g \in G$ and $x \in \tilde{\D}$ we have $g \bigl[ \tilde{\Y}(x) \bigr] = \tilde{\Y} \bigl[ g(x) \bigr]$. First, note that, by \eqref{E:Y.tilde(x).defn}
 and the definition, \eqref{E:x.in.bar.f.ij.in.Uij}, of $i_{N_{1}+1}, \ldots, i_{N_{2}}$: 
    \begin{multline*}
      \tilde{\Y}(x) = \Y(x) \setminus 
        \left( \bigcup_{j = N_{1}+1}^{N_{2}} supp \, \bar{f}_{i_{j}} \right) \\
          = \Y(x) \setminus \left( \bigcup_{(supp \, \bar{f}_{i}) \cap \Y(x) \neq \varnothing; 
            x \notin supp \, \bar{f}_{i}} 
              supp \, \bar{f}_{i} \right) 
               = \Y(x) \setminus \left( \bigcup_{x \notin 
                 supp \, \bar{f}_{i}} supp \, \bar{f}_{i} \right).
    \end{multline*}
Therefore, by \eqref{E:g.commutes.w/.set.ops} and \eqref{E:Y.commutes.w/.G}, it suffices to show 
    \begin{equation}  \label{E:g.supp.union.commute}
        g \left( \bigcup_{x \notin supp \, \bar{f}_{i}} supp \, \bar{f}_{i} \right)
         = \bigcup_{g(x) \notin supp \, \bar{f}_{i}} supp \, \bar{f}_{i}.
    \end{equation}
The proof is similar to that of \eqref{E:Y.commutes.w/.G}. By \eqref{E:g.and.supp.fi}, 
we have $g(supp \, \bar{f}_{i}) = supp \, \bar{f}_{i} \circ g^{-1}$. Now, by \eqref{E:G.equivariance.of.f.bar}, for every $i$ there exists a $j$ s.t.\ $\bar{f}_{i} \circ g^{-1} = \bar{f}_{j}$, so, by \eqref{E:g.and.supp.fi}, $g(supp \, \bar{f}_{i}) = supp \, \bar{f}_{j}$. In particular, $x \notin supp \, \bar{f}_{i}$ 
if and only if $g(x) \notin supp \, \bar{f}_{j}$. Thus, by \eqref{E:g.commutes.w/.set.ops}, \eqref{E:g.and.supp.fi}, and \eqref{E:G.equivariance.of.f.bar}, 
    \begin{multline} \label{E:g.union.x.supps.=.union.g(x).supps}
      g \left( \bigcup_{x \notin supp \, \bar{f}_{i}} supp \, \bar{f}_{i} \right)
        = \bigcup_{x \notin supp \, \bar{f}_{i}} g(supp \, \bar{f}_{i}) \\
        = \bigcup_{g(x) \notin g(supp \, \bar{f}_{i})} g(supp \, \bar{f}_{i})
         \subset \bigcup_{g(x) \notin supp \, \bar{f}_{j}} supp \, \bar{f}_{j}.
    \end{multline}
 That is one direction of the equality \eqref{E:g.supp.union.commute}, to get the other direction, let $h \in G$ be arbitrary and replace $x$ in \eqref{E:g.union.x.supps.=.union.g(x).supps} by $h(x)$ and $g$ by $g^{-1}$. This yields
    \begin{equation*} 
      g^{-1} \left( \bigcup_{h(x) \notin supp \, \bar{f}_{i}} supp \, \bar{f}_{i} \right)
         \subset \bigcup_{g^{-1} \circ h(x) \notin supp \, \bar{f}_{j}} supp \, \bar{f}_{j}.
    \end{equation*}
 Then replace $h$ by $g$ to get
    \begin{equation*}
      g^{-1} \left( \bigcup_{g(x) \notin supp \, \bar{f}_{i}} supp \, \bar{f}_{i} \right)
         \subset \bigcup_{x \notin supp \, \bar{f}_{j}} supp \, \bar{f}_{j}.
    \end{equation*}
 Combining this with \eqref{E:g.union.x.supps.=.union.g(x).supps} yields \eqref{E:g.supp.union.commute} and proves the claim \eqref{E:Y.tilde.G.equivar}.

\emph{Claim:} $\Xi$ is invariant under $G$ action. Let $x \in \tilde{\D}$ and $g \in G$. If $m = 1$, then $G$ is the trivial group and we are done. Suppose $m > 1$. 
For every $k = 1, 2, \ldots, m$ let $\ell = \pi_{g}(k)$ ($\ell= 1, \ldots, m$) satify 
$g_{\ell} = g^{-1} \circ g_{k}$. Then $\pi := \pi_{g}$ is a permutation 
of $\{ 1, 2, \ldots, m \}$ and by \eqref{E:G.invariance.of.X's}  
	\begin{equation}  \label{E:Xk.=.g(Xpik)}
		X_{jm+k} = g_{k}(X_{jm+1}) =g \circ g_{\pi(k)}(X_{jm+1}) = g(X_{jm+\pi(k)}).
	\end{equation}
Hence, by \eqref{E:Xi.is.defined.and.cont}, \eqref{E:G.equivariance.of.f.bar}, \eqref{E:Xk.=.g(Xpik)}, and $G$ invariance of $\Phi$, we have
 \begin{small}
   \begin{align*}
      \Xi \bigl[ g(x) \bigr]   \\
        \begin{split}
	= \gamma \biggl[ \Bigl( \ldots, \bar{f}_{jm} \bigl[ g(x) \bigr], 
    		\bar{f}_{jm+1} \bigl[ g(x) \bigr], \bar{f}_{jm+2} \bigl[ g(x) \bigr], 
               	\ldots, \bar{f}_{(j+1)m} \bigl[ g(x) \bigr], 
              	   \bar{f}_{j(m+1)+1} \bigl[ g(x) \bigr], \ldots \Bigr), \\
	   \Bigl( \ldots, \Phi(X_{jm}), \Phi(X_{jm+1}), \Phi(X_{jm+2}), \ldots, \Phi(X_{(j+1)m}), 
	                       \Phi(X_{(j+1)m+1}) , \ldots \Bigr)  \biggr]
	\end{split} \\
        \begin{split}
		= \gamma \biggl[ \Bigl( \ldots, \bar{f}_{(j-1)m+\pi(m)} (x), \bar{f}_{jm+\pi(1)} (x), 
				\bar{f}_{jm+\pi(2)} (x), 
               \ldots, \bar{f}_{jm + \pi(m)} (x), 
              	   \bar{f}_{j(m+1)+\pi(1)} (x), \ldots \Bigr), \\
	   \Bigl( \ldots, \Phi(X_{jm}), \Phi(X_{jm+1}), \Phi(X_{jm+2}), 
	     \ldots, \Phi(X_{(j+1)m}), 
	   			\Phi(X_{(j+1)m+1}) , \ldots \Bigr) , \ldots \Bigr)  \biggr]
	\end{split} \\
        \begin{split}
		= \gamma \biggl[ \Bigl( \ldots, \bar{f}_{(j-1)m+\pi(m)} (x), \bar{f}_{jm+\pi(1)} (x), 
				\bar{f}_{jm+\pi(2)} (x), 
               \ldots, \bar{f}_{jm + \pi(m)} (x), 
              	   \bar{f}_{j(m+1)+\pi(1)} (x), \ldots \Bigr), \\
	   \Bigl( \ldots, \Phi(X_{(j-1)m+\pi(m)}), \Phi(X_{jm+\pi(1)}), \Phi(X_{jm+\pi(2)}), \ldots, 
	   		\Phi(X_{jm + \pi(m)}), \Phi(X_{(j(m+1)+\pi(1)}) , \ldots \Bigr)  \biggr].
	\end{split} 
   \end{align*}
 \end{small}
 Hence, since $\gamma$ is commutative if $m > 1$, we have 
$\Xi \bigl[ g(x) \bigr]  = \Xi (x)$, i.e., 
    \begin{equation}  \label{E:Xi.is.G-invar}
      \Xi \text{ is $G$-invariant}
    \end{equation}
and the claim is proved. For the part \eqref{I:cont.approx.to.Phi.exists} of the theorem we may take 
$\Omega := \Omega_{\gamma} := \Xi$. 

\emph{Proof of part \eqref{I:Omega.Phi.agree.on.Pf} of the theorem:} Suppose $\Pf \subset \D$ is closed, $G(\Pf) = \Pf$, $\Pf \cap \D'$ is dense in $\Pf$, and $\Phi \restriction_{\Pf \cap \D'}$ has a unique continuous extension, $\Theta$, to $\Pf$. Suppose $\Rcl$ is a $G$-invariant neighborhood of $\Pf$ and there exists a $G$-equivariant retraction 
$\rho : \Rcl \to \Pf$.
We show that, after possible tweaking of $\Xi$ in $\Rcl$, we arrive at a map $\Omega : \tilde{\D} \to \F$ that agrees with $\Phi$ 
on $\D' \cap \Pf$. The idea of the proof is as follows. First, we construct a continuous map $\Psi: \Rcl \to \F$ that agrees with $\Theta$ on $\Pf$. Then, using Urysohn's lemma, we define $\Omega$ to to equal $\Xi$ off $\Rcl$ and, on $\Rcl$ to be a convex combination of $\Xi$ and $\Psi$.

That $\rho$ is a retraction means that $\rho(x) = x$ if $x \in \Pf$. $\rho$, of course, is not one-to-one so does not have an inverse. Still, we show that $G$-equivariance of $\rho$ implies that of its inverse set mapping, specifically,
    \begin{equation*}
      \rho^{-1} \bigl[ g(x) \bigl]  = g \bigl[ \rho^{-1}(x) \bigr], \quad g \in G , 
        \; x \in \Pf . 
    \end{equation*}
To see this, let $g \in G$ and $x \in \Pf$ and suppose 
$y \in g \bigl[ \rho^{-1}(x) \bigr]$. This is true if and only if $\rho \circ g^{-1} (y) = x$. By equivariance of $\rho$, this is true if and only if $g^{-1} \circ \rho(y) = x$, which is equivalent to $y \in \rho^{-1} \bigl[ g(x) \bigl]$. 

The unique continuous extension, $\Theta$, is $G$-invariant (since $\Phi$ is). Define
    \begin{equation}  \label{E:Psi=Theta.circ.rho}
            \Psi(x) := \Theta \circ \rho(x), \quad (x \in \Rcl).
    \end{equation}
Thus, $\Psi$ is continuous, $G$-invariant, and $\Psi(x) = \Theta(x)$ if $x \in \Pf$. Hence, since $\Pf \cap \D'$ is dense in $\Pf$ and $\Theta$ is the continuous extension of 
$\Phi \restriction_{\Pf \cap \D'}$ to $\Pf$ we have the following. Recall \eqref{E:Y.tilde(x).defn}. Let $x \in \tilde{\D} \cap \Pf$, $x', x'' \in \tilde{\Y}(x)$ 
with $x' \in \Pf$ and $x'' \in \tilde{\D}$. Then, by \eqref{E:Phi.Xi.x'.in.same.V}, 
there exists $V \in \msf{V}$ s.t.\
    \begin{equation}  \label{E:Theta.Xi.in.same.V}
	      \Theta(x') = \lim_{y \to x', \, y \in \D' \cap \Pf} \Phi(y) 
	        \in \overline{ \Phi \bigl[ \tilde{\Y}(x)  \cap \D' \bigr] } \subset V 
	          \text{ and } \Xi(x'') \in V.
    \end{equation}

Let
   \[
      \mcl{W} := \bigcup_{x \in \tilde{D} \cap \Pf} \tilde{\Y}(x) \cap 
        \rho^{-1} \bigl[ \tilde{\Y}(x) \cap \Pf \bigr] .
   \]
$\mcl{W}$ is open and 
    \begin{equation*}
      \tilde{\D} \cap \Pf \subset \mcl{W} \subset \Rcl .
    \end{equation*} 
\emph{Claim:} $\mcl{W}$ is $G$-invariant. For let $x \in \tilde{\D} \cap \Pf$, $g, h \in G$, and 
$y \in \Rcl \cap \tilde{\D}$. Then, by \eqref{E:Y.tilde.G.equivar}, \eqref{E:g.commutes.w/.set.ops}, $G$-invariance of $\Pf$, and $G$-equivariance of $\rho$, we have
    \begin{align} \label{E:y.in.rho.invrs.Y.tilde.g(x)}
         y \in \rho^{-1} \Bigl( \tilde{\Y} \bigl[ g(x) \bigr] \cap \Pf \Bigr)  &\text{ if and only if } 
           \rho(y) \in g \bigl[ \tilde{\Y}(x) \bigr] \cap \Pf 
             = g \bigl[ \tilde{\Y}(x) \cap \Pf \bigr] \notag \\
         &\text{ if and only if } \rho \circ g^{-1}(y) = g^{-1} \circ \rho(y) \in \tilde{\Y}(x) \cap \Pf \\
         &\text{ if and only if }  g^{-1}(y) \in \rho^{-1} \bigl[ \tilde{\Y}(x) \cap \Pf \bigr]. \notag
    \end{align}
Suppose $y \in \mcl{W}$. Then, by $G$-invariance of $\tilde{D} \cap \Pf$, there exists 
$x \in \tilde{D} \cap \Pf$ s.t.\ $y \in \tilde{\Y} \bigl[ g(x) \bigr] 
\cap \rho^{-1} \Bigl( \tilde{\Y} \bigl[ g(x) \bigr]  \cap \Pf \Bigr)$. 
Then, by \eqref{E:y.in.rho.invrs.Y.tilde.g(x)} and \eqref{E:Y.tilde.G.equivar}, and \eqref{E:g.commutes.w/.set.ops},
we get that 
    \begin{equation*}
        y \in \tilde{\Y} \bigl[ g(x) \bigr] \cap g \Bigl( \rho^{-1} \bigl[ \tilde{\Y} (x) 
          \cap \Pf \bigr] \Bigr) 
          = g \bigl[ \tilde{\Y}(x) \bigr] \cap g \Bigl( \rho^{-1} \bigl[ \tilde{\Y} (x) 
            \cap \Pf \bigr] \Bigr) \subset g(\mcl{W}).
    \end{equation*} 
I.e., $\mcl{W} \subset g(\mcl{W})$. Replacing $g$ by $g^{-1}$ shows that $\mcl{W} \supset g(\mcl{W})$.
This proves the claim that $\mcl{W}$ is $G$-invariant.

Let $y \in \mcl{W} \cap \tilde{\D}$. Then there exists $x \in \tilde{\D} \cap \Pf$ s.t.\ $y \in \tilde{\Y}(x)$ and there exists $x' \in \tilde{\Y}(x) \cap \Pf$ s.t.\ $\rho(y) = x'$. (Note that $x'$ may or may not be in $\tilde{\D}$.) Hence, by \eqref{E:Psi=Theta.circ.rho}, 
$\Psi(y) = \Theta(x')$. Therefore, by \eqref{E:Theta.Xi.in.same.V}, there exists 
$V \in \msf{V}$ s.t.\ 
   \begin{equation} \label{E.Phi.Psi.Xi.in.V}
      \Psi(y), \, \Xi(y) \in V \text{ if } y \in \mcl{W} \cap \tilde{\D} .
   \end{equation}

Now, $\D \setminus \mcl{W}$ and, by assumption, $\Pf$ are closed in $\D$. 
Hence, $\Pf \cap \tilde{\D} \subset \mcl{W}$ and $\tilde{\D} \setminus \mcl{W}$ are disjoint relatively closed subsets 
of the metric space $\tilde{\D}$. Therefore, there exists a neighborhood 
$\mcl{W}' \subset \tilde{\D}$ of $\Pf \cap \tilde{\D}$ s.t.\ 
$\overline{\mcl{W}'} \subset \Rcl \cap \mcl{W} \cap \tilde{\D}$, 
where $\overline{\mcl{W}'}$ is 
the closure of $\mcl{W}$ in $\tilde{\D}$. We may assume $\mcl{W}'$ is $G$-invariant. 
If it is not, just replace $\mcl{W}'$ by $\bigcap_{g \in G} g(\mcl{W}')$, a finite intersection of open sets. We continue to have 
$\Pf \cap \tilde{\D} \subset \mcl{W}' \subset \overline{\mcl{W}'} \subset \mcl{W}$.

Therefore, by Urysohn's lemma (Simmons \cite[Theorem A, p.\ 135]{gfS63}) there is a continuous function 
$\mu: \tilde{\D} \to [0,1]$ that is 1 on  $\Pf \cap \tilde{\D}$ and 0 off $\mcl{W}'$. Now, $\tilde{\D}$, $\Pf$, and $\mcl{W}'$ are all $G$-invariant. Therefore, replacing $\mu$ by $m^{-1} \sum_{g \in G} \mu \circ g$, we may assume $\mu$ is $G$-invariant. Let  
   \[
      \Omega(x) = 
         \begin{cases}
            \gamma \Bigl[ \bigl( \mu (x), 1 - \mu(x) \bigr), \bigl( \Psi(x), \Xi(x) \bigr) \Bigr],
                &\text{if } x \in \mcl{W}' \\ 
            \Xi(x), &\text{if } x \in \tilde{\D} \setminus \mcl{W}'.
         \end{cases}
   \]
By \eqref{E.Phi.Psi.Xi.in.V} and the fact that $\gamma$ is a convex combinationn function 
on $\msf{V}$, $\Omega$ is defined everywhere on $\tilde{\D}$. By \eqref{E:Xi.is.defined.and.cont} and \eqref{E:Psi=Theta.circ.rho}, it is continuous. 
By \eqref{E:Xi.is.G-invar} and \eqref{E:Psi=Theta.circ.rho} again, it is $G$-invariant 
on $\tilde{\D}$. And, by \eqref{E:Psi=Theta.circ.rho} yet again, $\Omega = \Theta$ 
on $\tilde{\D} \cap \Pf$. Since 
$\tilde{\D} \setminus \Rcl \subset \tilde{\D} \setminus \mcl{W}'$,  we have 
$\Omega = \Xi$ off $\Rcl$.
  \end{proof}

The preceding proof raises Axiom of Choice issues. See remark \ref{R:axiom.of.choice}.

\section{Approximate continuity on $\Pf$, discrete maps}  \label{SS:approx.cont.} So far we have developed the ``sales pitch'' idea (remarks \ref{R:sales.pitch} and \ref{R:extended.sales.pitch}) to show that if the set of singularities on or near $\T = \Pf$ is small then we can infer something about singularities elsewhere. Instead of restricting the volume of the singular set on or near $\T$ we can make a similar inference by restricting the \emph{severity} of singularities locally around $\T$. One way of doing this is a trivial application of the ``severity trick'' (remark \ref{R:severity.trick}). This will be exemplified in proposition \ref{P:sing.codim.in.plane-fitting}, and at various points in chapters \ref{Chptr:spherical.location}, \ref{Chptr:aug.direct.mean}, and \ref{Chptr:robst.loc.on.circle}.

In this section we develop another version of the severity-instead-of-dimension idea. 

Using theorem \ref{T:if.lin.combo.on.F.then.can.rstrct.to.bad.sings} we prove, presently, two results that allow weakening of hypotheses \ref{Hyp:S.cap.T.small} and \ref{Hyp:extend} of theorem \ref{T:Phi.star.Hr.contains.Theta.star.Hr}. In the following, we regard $\Phi : \D' \to \F$ as only defined on $\D'$ so inverse images under $\Phi$ lie entirely in $\D'$. Find the proof below after remark \ref{R:discreteness}. 
  \begin{prop}  \label{P:V1.V2.V.homotop}
      Let $\D$, $G$, $\D' \subset \D$, and $\Phi : \D' \to \F$ be as in theorem 
      \ref{T:if.lin.combo.on.F.then.can.rstrct.to.bad.sings} up through part \eqref{I:cont.approx.to.Phi.exists}. 
      We do \emph{not} require the extra assumptions in theorem \ref{T:if.lin.combo.on.F.then.can.rstrct.to.bad.sings} part \ref{I:Omega.Phi.agree.on.Pf}.
      Suppose $\msf{V}$, $\msf{V}_{1}$, 
      and $\msf{V}_{2}$ are covers of $\F$. 
      Assume
          \begin{equation}  \label{E:V1.union.V2.in.V}
             \text{If } V_{i} \in \msf{V}_{i} \; (i = 1,2) 
               \text{ and } V_{1} \cap V_{2} \neq \varnothing, 
                 \text{ then there exists } V \in \msf{V} \text{ s.t.\ } 
                   V_{1} \cup V_{2} \subset V.
          \end{equation}
      (In particular, $\msf{V}_{i}$ ($i=1,2$) are refinements of $\msf{V}$. 
      $V_{1} = V_{2}$ is possible.)
      Let $\gamma$ be a convex combination function on $\msf{V}$ and 
      let $\gamma_{2}$ be a convex combination function on $\msf{V}_{2}$. 
      Let $\Pf \subset \D$ and suppose 
         \begin{equation} \label{E:S.V2..cap.Pf.empty}
		\Ss^{\msf{V}_{2}} \cap \Pf = \varnothing 
         \end{equation}
      and there is a continuous function $\Sigma : \Pf \to \F$ s.t., 
         \begin{equation}  \label{E:Phix.Sigma(x).in.V1}
		\text{For every } x \in \Pf \text{ there exists } 
		  V_{1} \in \msf{V}_{1} \text{ s.t.\ }
		  \Sigma(x) \in V_{1}   
		    \text{ and } x \in \overline{\Phi^{-1}(V_{1}) \cap \Pf \cap \D'} . 
	\end{equation}
Here, as usual $\overline{\Phi^{-1}(V_{1}) \cap \Pf \cap \D'}$
denotes the closure of $\Phi^{-1}(V_{1}) \cap \Pf \cap \D'$ in $\D$. 
      Let $\Omega_{\gamma_{2}}: \D \setminus \Ss^{\msf{V}_{2}} \to \F$ be the $G$-invariant map as in theorem \ref{T:if.lin.combo.on.F.then.can.rstrct.to.bad.sings} part \ref{I:cont.approx.to.Phi.exists} 
      with $\gamma_{2}$ in place of $\gamma$ and $\msf{V}_{2}$ in place of $\msf{V}$ there. 
      Then $\Omega_{\gamma_{2}}$ 
      is defined and continuous on all of $\Pf$ and the restriction 
      of $\Omega_{\gamma_{2}}$ 
      to $\Pf$ is homotopic to $\Sigma$. 
   \end{prop}
 
Think of the function $\Sigma$ as a continuous ``standard'," as defined in section \ref{SS:calibration}. Note that closure of subsets of $\Pf$ in $\D$ is the same as closure 
in $\Pf$ because we retain the assumption in theorem \ref{T:if.lin.combo.on.F.then.can.rstrct.to.bad.sings} that $\Pf$ itself is closed in $\D$. Observe that \eqref{E:Phix.Sigma(x).in.V1} implies 
\emph{$\D' \cap \Pf$ is dense in $\Pf$.} But $\Phi$ is not required to satisfy \textbf{hypothesis \ref{Hyp:S.cap.T.small}} of theorem \ref{T:Phi.star.Hr.contains.Theta.star.Hr}.
Similarly, in the proposition it is \emph{not} required that $\Phi$ satisfy \textbf{hypothesis \ref{Hyp:extend}} of theorem \ref{T:Phi.star.Hr.contains.Theta.star.Hr}. Here the \emph{restriction} $\Phi \restriction_{\T \setminus \Ss}$ is allowed have singularities. 

Notice that it is not required that $\msf{V}_{1}$ be a refinement of $\msf{V}_{2}$ or \emph{vice versa}. However, in light of \eqref{E:S.V2..cap.Pf.empty}, in order that \eqref{E:Phix.Sigma(x).in.V1} not be vacuous, the sets in $\msf{V}_{1}$ must be smaller in some sense than those in $\msf{V}_{2}$.
 
  \begin{remark} \label{R:approx.extension.of.Phi.T}
Assume the hypotheses of proposition \ref{P:V1.V2.V.homotop} hold. Consider the statement
     \begin{multline}  \label{E:Phil.cont.on.T}
        \text{For every } x \in \Pf \text{ and \emph{any} neighborhood, } V_{1} , 
          \text{ of } \Sigma(x) \\
            \text{ (not restricted to elements of } \msf{V}_{1}), 
              \text{there exists a neighborhood } \mcl{W} \text{ of } x  \\
                \text{ s.t.\  } \Phi ( \mcl{W} \cap \D' \cap \Pf ) \subset V_{1}.
    \end{multline}
\eqref{E:Phil.cont.on.T} says that $\Phi(x') \to \Sigma(x)$ as $x' \to x$ through $\D' \cap \Pf$ and $\Phi(x) = \Sigma(x)$ if $x \in \D'$. So, according to \eqref{E:Phil.cont.on.T}, 
$\Phi \restriction_{\Pf \cap \D'}$ has a continuous extension, \emph{viz.}\ $\Sigma$, on $\Pf$. Now,
\eqref{E:Phix.Sigma(x).in.V1} resembles \eqref{E:Phil.cont.on.T} except in \eqref{E:Phix.Sigma(x).in.V1} $V_{1}$ must belong to $\msf{V}_{1}$ and we require 
$x \in \overline{\Phi^{-1}(V_{1}) \cap \Pf}$. So we might interpret \eqref{E:Phix.Sigma(x).in.V1} as saying that $\Phi \restriction_{\Pf \cap \D'}$ has a continuous approximate extension on $\Pf$, \emph{viz.} $\Sigma$, where ``approximate'' is defined by $\msf{V}_{1}$.

If the map $\Sigma : \Pf \to \F$ is a standard for the problem (section \ref{SS:calibration}), then \eqref{E:Phix.Sigma(x).in.V1} says that $\Phi$ is approximately calibrated. 
  \end{remark}
  
  \begin{remark}
\eqref{E:Phix.Sigma(x).in.V1} does not require that every $x \in \Pf$  has a neighborhood $\mcl{W}$ s.t.\ $\overline{\Phi ( \mcl{W} \cap \D' )}$ lies in some 
$V_{1} \in \msf{V}_{1}$. So it does not preclude 
$\Ss^{\msf{V}_{1}} \cap \Pf \neq \varnothing$. 
Here is a fable illustrating this point. Let $\Pf = [-1,1]$ and $\F = \RR$. 
For every $\ell = 0, 1, 2, \ldots$ let 
$\Phi(\pm 2^{-\ell}) = \ell \pmod 2$. So if $\ell$ is odd, $\Phi(\pm 2^{-\ell}) = 1$. If $\ell$ is even, $\Phi(\pm 2^{-\ell}) = 0$. For $x \in [-1, 1] \setminus \{0\}$ define $\Phi(x)$ by linear interpolation between successive values of $\pm 2^{-\ell}$. Let $\msf{V}_{1}$ be the collection of all open intervals of length 1. Let $\msf{V}_{2}$ be the collection of all open intervals of length 2. Then 0 is a $\msf{V}_{1}$-severe singularity of $\Phi$ w.r.t.\ 
$\D' := \D \setminus \{0\}$, where $\D$ is some superset of $\Pf$. But $0$ is not 
$\msf{V}_{2}$-severe. 

Let $V_{1} := (0,1) \in \msf{V}_{1}$. Then we have 
$\Phi^{-1}(V_{1}) \cap \Pf \cap \D' = \Pf \setminus \bigl( \{0\} \cup \{ \pm 2^{-\ell} : \ell = 0, 1, 2, \ldots \} \bigr)$. Therefore, 
$\overline{\Phi^{-1}(V_{1}) \cap \Pf \cap \D'} = \Pf$. Hence, if we define 
$\Sigma \equiv 1/2$ and let $V_{1} = (0,1) \in \msf{V}_{1}$ then \eqref{E:Phix.Sigma(x).in.V1} holds for every $x \in \Pf$. 
  \end{remark}

  \begin{remark}[Using proposition \ref{P:V1.V2.V.homotop}] \label{R:using.not.cont.prop}
Here we suggest how proposition \ref{P:V1.V2.V.homotop} can be used to prove that 
$\Ss^{\msf{V}_{2}}$ is not empty. Suppose $\T \subset \D$ is a compact manifold (in the relative topology) so \textbf{hypothesis \ref{Hyp:T.manif}} of theorem \ref{T:Phi.star.Hr.contains.Theta.star.Hr} holds. Suppose \textbf{hypothesis 
\ref{Hyp:r.integer}} holds.

Suppose the hypotheses of proposition \ref{P:V1.V2.V.homotop} are satisfies with 
$\Pf = \T$. Redefine $\Phi := \Omega_{\gamma_{2}}$ and $\Ss' := \Ss^{\msf{V}_{2}}$ as in proposition \ref{P:V1.V2.V.homotop} 
for the moment. Then, by \eqref{E:SV.is.closed} and assumptions inherited indirectly from theorem \ref{T:if.lin.combo.on.F.then.can.rstrct.to.bad.sings}, \textbf{hypothesis \ref{Hyp:S.cap.T.closed}} of theorem \ref{T:Phi.star.Hr.contains.Theta.star.Hr} holds. Since 
$\Ss' \cap \T = \Ss^{\msf{V}_{2}} \cap \T = \varnothing$, by \eqref{E:S.V2..cap.Pf.empty}, we have that \textbf{hypotheses \ref{Hyp:S.cap.T.small} and \ref{Hyp:extend}} of theorem \ref{T:Phi.star.Hr.contains.Theta.star.Hr} holds with $\Theta$ equal to the restriction $\Omega_{\gamma_{2}} \restriction_{\T}$. 

Suppose $\Sigma_{\ast} : H_{r}(\T) \to H_{r}(\F)$ is nontrivial. We then have by proposition \ref{P:V1.V2.V.homotop} with $\Pf := \T$ that \eqref{E:nontriv.r-dim.homol} also holds for $\Phi = \Omega_{\gamma_{2}}$. That is in addition to all the hypotheses of theorem \ref{T:Phi.star.Hr.contains.Theta.star.Hr}. This is a promising state of affairs for proving things about $\Ss' = \Ss^{\msf{V}_{2}}$.

An interesting question is, is there an example of a $\Phi$ for which the preceding argument shows $\Phi$ has $\msf{V_{2}}$-severe singularities, but $\Phi$ has no 
$\msf{V}$-severe singularities?
  \end{remark} 

Chapter \ref{Chptr:sings.in.plane.fit} is about singularity in plane-fitting (section \ref{SS:examples}, example \ref{Exmpl:plane.fitting} in section \ref{SS:examples}). In that setting $H_{r}(\D)$ is trivial (section \ref{SS:D.T.plane.fit}) and convex combination combination functions on $\F$ exist (proposition \ref{P:sing.codim.in.plane-fitting}). Therefore, it seems that proposition \ref{P:V1.V2.V.homotop}
can be applied to plane-fitting.  

  \begin{remark} \label{R:improvments.over.main.thm}
Proposition \ref{P:V1.V2.V.homotop} improves upon theorem \ref{T:Phi.star.Hr.contains.Theta.star.Hr} as follows.
	\begin{enumerate}
		\item We may replace the partially closed superset, $\Ss' \supset \Ss$, 
		of the singular set $\Ss$ in hypothesis \ref{Hyp:S.cap.T.closed} 
		of theorem \ref{T:Phi.star.Hr.contains.Theta.star.Hr} 
		  by some $\Ss^{\msf{V}_{2}} \subset \Ss$.
		\item We may drop \textbf{hypothesis \ref{Hyp:S.cap.T.small}}.
		\item We may drop the ``infinite resolution'' \textbf{hypothesis \ref{Hyp:extend}} 
		of theorem \ref{T:Phi.star.Hr.contains.Theta.star.Hr} and replace it by the 
		``finite resolution'' hypothesis \eqref{E:Phix.Sigma(x).in.V1}. 
		(See remark \ref{R:approx.extension.of.Phi.T}.)  
	\end{enumerate}
  \end{remark}

  \begin{proof}[Proof of proposition \ref{P:V1.V2.V.homotop}]
We apply theorem \ref{T:if.lin.combo.on.F.then.can.rstrct.to.bad.sings}, part \eqref{I:cont.approx.to.Phi.exists} with $\msf{V}_{2}$ in place of $\msf{V}$ and 
$\mathfrak{D} = \{ \D \}$. 
Denote by $\Omega_{\gamma_{2}}$ the function $\Omega$ promised by part \eqref{I:cont.approx.to.Phi.exists} of the theorem. By \eqref{E:S.V2..cap.Pf.empty}, 
$\Ss^{\msf{V}_{2}} \cap \Pf = \varnothing$ so 
$\Pf \subset \tilde{\D}_{2} := \D \setminus \Ss^{\msf{V}_{2}}$. 
Let $x \in \Pf$. 
By theorem \ref{T:if.lin.combo.on.F.then.can.rstrct.to.bad.sings}, part \eqref{I:cont.approx.to.Phi.exists}, there exists  a neighborhood $\clU$ of $x$ 
and $V_{2} \in \msf{V}_{2}$ s.t.\ $\overline{\Phi(\clU \cap \D')} \subset V_{2}$ and 
$\Omega_{\gamma_{2}}(x) \in  V_{2}$. (We have already observed that \eqref{E:Phix.Sigma(x).in.V1} implies $\D' \cap \Pf$ is dense in $\Pf$.) Moreover, 
$\Omega_{\gamma_{2}}$ is continuous on 
$\tilde{\D}_{2}$.  
In particular, 
    \begin{equation*}
      \Omega_{\gamma_{2}} \text{ is continuous on } \Pf .
    \end{equation*}

Let $V_{1} \in \msf{V}_{1}$ be as in \eqref{E:Phix.Sigma(x).in.V1}. 
Thus, $\Sigma(x) \in V_{1}$ and $x \in \overline{\Phi^{-1}(V_{1}) \cap \Pf \cap \D'}$.
Hence, there exists $x' \in \Pf$ arbitrarily close to $x$ s.t.\ 
$x' \in \Phi^{-1}(V_{1}) \cap \Pf \cap \D'$. We may take $x' \in \clU \cap \Pf$. Thus, 
$x' \in \Phi^{-1}(V_{1}) \cap \Pf \cap \clU \cap \D' $. I.e., $\Phi(x') \in V_{1} \cap V_{2}$ so $V_{1} \cap V_{2} \neq \varnothing$. Hence (see next paragraph), by \eqref{E:V1.union.V2.in.V}, there exists $V \in \msf{V}$ s.t.\ $V_{1} \cup V_{2} \subset V$. I.e., $\Omega_{\gamma_{2}}(x), \, \Sigma(x) \in V$. Thus, by definition \ref{D:convex.combo.fn}, the following function is well-defined and continuous.
   \[
      H(x,t) = \gamma \Bigl[ (t, 1-t), \bigl( \Sigma(x), \Omega_{\gamma_{2}}(x) \bigr) \Bigr], 
        \quad x \in \Pf.
   \]
$H$ is the desired homotopy. (In particular, see \eqref{E:drop.0.coefs} and part \eqref{I:1.is.identity} of definition \ref{D:convex.combo.fn}.)

(Suppose we tried to go through this proof with $V \in \msf{V}$ in place of 
$V_{2} \in \msf{V}_{2}$. That works and we get an $\Omega$ continuous on $\Pf$. But we run into trouble when we reach the ``Hence'': There might not be a set in $\msf{V}$ containing $V_{1} \cup V$. Thus, \emph{prima facie} we do not get a homotopy between the restriction of $\Omega_{\gamma_{2}}$ and $\Sigma$.)
  \end{proof}

   \begin{remark}[Discretness] \label{R:discreteness}
These results may allow our theory to be applied to cases like discrete maps in which precision is finite. The level of precision is governed by the cover $\msf{V}_{1}$. 
If $\Phi$ is ``almost'' continuous then one can get away with using open sets 
in $\msf{V}_{1}$ with small diameter, which allows the sets in $\msf{V}_{2}$ to be bigger, and therefore the singularities in $\Ss^{\msf{V}_{2}}$ to be more severe. Thus, not surprisingly, having $\Phi$ close to continuous on $\D' \cap \Pf$ is a stronger condition, in the sense of implying more severe singularity, than the opposite. 

Suppose the codomain of $\Phi$, call it $\F_{discrete}$, is a discrete metric space, for example a lattice. Then unless $\Phi$ is trivial, it will have singularities, jumps. We may tolerate jumps, providing they are not too big, but big jumps may be distressing. We may be able to use proposition \ref{P:V1.V2.V.homotop} to analyze the problem of large jumps in $\Phi$. Suppose we are able to isometrically imbed $\F_{discrete}$ into a connected metric space $\F$ that has some interesting topology. Let $\msf{V}_{1}$ be an open cover of $\F$ that defines what it means for a jump to be ``small'': Both ends of the jump belong to some $V_{1} \in \msf{V}_{1}$. Suppose we can find other open covers $\msf{V}$ and 
$\msf{V}_{2}$, having convex combination functions, s.t.\ \eqref{E:V1.union.V2.in.V} holds. If the sets in $\msf{V}_{2}$ are much ``bigger'' than those in $\msf{V}_{1}$, we might view jumps that escape any $V_{2} \in \msf{V}_{2}$ as ``large''. 

Suppose the domain of $\Phi$ is also discrete or disconnected. (So we have to relax \eqref{E:D.metric.F.normal} here.) A class of data maps of this sort is described in chapter \ref{Chptr:linear.classification}. (But the discreteness aspect of such data maps is not considered there.) Call the codomain $\D_{discrete}$. We may be able to isometrically imbed $\D_{discrete}$ into a space $\D$ satisfying \eqref{E:D.metric.F.normal}. Now we can extend $\Phi$ to a dense subset of $\D$ as follows. Let $\mbf{T}$ be the Voronoi tessellation of $\D$ corresponding to $\D_{discrete}$ (Edelsbrunner and Harer \cite[p.\ 65]{hEjlH10.CompTopol}). If $\mbf{S}$ is the open Voronoi cell containing a point $x \in \D_{discrete}$, define $\Phi$ on the interior 
of $\mbf{S}$ to have the constant value $\Phi(x)$. Of course, the extended $\Phi$ will have singularities along the boundaries of the Voronoi cells, but we ignore those that are only 
$\msf{V}_{1}$-severe. We might be able to apply proposition \ref{P:V1.V2.V.homotop}, 
to the extended map $\Phi : \D \partlyto \F$. 

Another situation in which we may be able to exploit proposition \ref{P:V1.V2.V.homotop}, is when the behavior of the map $\Phi$ is difficult to analyze. This is the case if $\Phi$ is non-algorithmic (section \ref{S:topology}) 
Also if $\Phi$ is a deep neural net, for example. (See section \ref{S:DNN}.) But even when $\Phi$ has a compact algorithmic description it might be difficult to analyze. 
In these difficult cases there may be a small, interesting test pattern set $\T \subset \D$ near which it is possible to understand the behavior of $\Phi$.

But suppose a theoretical analysis of $\Phi$ even on a small $\T$ is difficult. 
Let $\T \subset \Pf \subset \D$ be compact manifolds and $\Sigma : \Pf \to \F$ a standard for the type of data analysis problem $\Phi$ is intended to solve (subsection \ref{SS:calibration}). 
We then find ourselves in a ``damned if you do and damned if you don't'' situation 
(remark \ref{R:Scylla.and.Charybdis}):
Either $\Phi$ has severe singularities or it is not calibrated, the latter being the more serious problem. We might be able to get insights experimentally: Choose a finite subset 
$X := \{x_{1}, \ldots, x_{n}\} \subset \T$ and compute $\Phi(x_{i})$ ($i =1, \ldots, n$). Or rather try to compute those values. For some $i$ we might find that $\Phi(x_{i})$ cannot be computed. That would be a bad sign. But suppose we are able to compute $\Phi$ on $X$.
We then compare $\Phi(x_{i})$ and $\Sigma(x_{i})$ ($i =1, \ldots, n$. If we are lucky, there is a fairly fine cover $\msf{V}_{1}$ -- ``fine'' means the sets in $\msf{V}_{1}$ are small -- s.t.\ \eqref{E:Phix.Sigma(x).in.V1} holds for $x \in X$. If it does that is consistent with $\Phi$ satisfying \eqref{E:Phix.Sigma(x).in.V1} on all of $\Pf = \T$. 

Approximate calibration on $X$ is encouraging, but does not \emph{prove} that \eqref{E:Phix.Sigma(x).in.V1} holds everywhere: It is possible that off $X$ the behavior 
of $\Phi$ is very bad. However, one might be able to obtain high confidence that $\Phi$ satisfies \eqref{E:S.V2..cap.Pf.empty} and \eqref{E:Phix.Sigma(x).in.V1}:
Assume \emph{a priori} that $\Phi \restriction_{\T}$ is a realization of a random process that seems reasonable. So the randomness is in $\Phi \restriction_{\T}$. Take $X$ to be a random, or at least very irregular, sample. Alternatively, assume $\Phi \restriction_{\T}$ belongs to some fairly broad class of maps and take $X$ to be random, so the randomness is in $X$. If $n$ is big enough one might be able to obtain high posterior probability of, or high confidence in, good behavior of $\Phi$ on $\T$. From that one might be able to confidently conclude that $\Phi$ has severe singularities, perhaps many of them. 

If the behavior of $\Phi$ on $X$ is inconsistent with good behavior 
of $\Phi \restriction_{\T}$, then one concludes, in a ''damned if you do and damned if you don't'' fashion (remark \ref{R:Scylla.and.Charybdis}) that $\Phi$ exhibits some kind of horrible behavior.
  \end{remark}

\section{Nondeterministic data maps}  \label{SS:nondeterministic.data.maps}
Our theory is deterministic but living organisms process data nondeterministically. 
Even formal data analysis by a statistician usually involves subjective inputs that can be thought of as random. Some algorithms, like bootstrapping and cross-validation (Efron and Tibshirani \cite{bErjT93.bootstrap}), or Markov Chain Monte Carlo (Robert and Casella \cite{cpRgC04.MonteCarlo}) involve the explicit introduction of randomness into the algorithm. Moreover, randomization is a useful theoretical technique for handling the singularities of hypothesis testing (Lehmann \cite[p.\ 71]{elL93.StatHyps}). So consider ``stochastic data maps'', $\Phi : \D' \to \mathbb{P}(\F)$, where
$\mathbb{P}(\F)$ is the set of all Borel probability measures on $\F$.
(So the stochastic element we are considering is in the \emph{analysis} of the data, not in the mechanism that generates the data, which may or may not be stochastic.)
And there is this: ``Modern algorithms often have a random component, such as random projections or random initial values in gradient descent and stochastic gradient descent.'' (Yu and Kumbier \cite[p.\  6]{bYkK2019.VeridicalDataScience}) 

An example where this may arise is as follows. Suppose $\Phi(x)$ is in general not an element of $\F$, but a well-behaved subset of $\F$. E.g., $\Phi(x)$ may be the result of an optimization 
and for some $x$ the set of optima may contain more than one element 
(e.g. lemma \ref{L:basic.LAD.soln.facts}). Let $\mu$ be a Borel measure on $\F$ that gives finite, positive mass to $\Phi(x)$ for all $x \in \D'$. Recall \eqref{E:integer.part.floor}. 
Then if $\Phi(x) = \msf{A}$ one might interpret $\Phi(x)$ as the probability measure 
$\mu(\msf{A})^{-1} ( \mu \lfloor \msf{A} )$, where $\mu \lfloor \msf{A}$
is the restriction of $\mu$ to $\msf{A}$. 

Clearly, this generalizes the concept of data map that we have been using.
Just consider the operation that takes a data map
$\Phi : \D' \to \F$ to $\tilde{\Phi} : \D' \to \mathbb{P}(\F)$ defined by $\tilde{\Phi}(x) =$ unit mass at $x$ for every $x \in \D'$.

However, even stochastic data analysis requires some level of consistency and so the concept of singularity even extends to the nondeterministic setting. A natural way to do this is to just put a topology on $\mathbb{P}(\F)$, e.g., the ``weak topology'' (Parthasarathy \cite{krP67.ProbOnMetricSpaces}, Billingsley \cite[pp.\ 237--239]{pB68.ConvProbMeas}), and apply the results we have obtained so far. 

Another approach is as follows. Let $\Phi$ be a stochastic data map and let $\msf{V}$ be an open cover of $\F$. 
Say that $x \in \D$ is a ``$\msf{V}$-severe singularity of $\Phi$'' and write $x \in \Ss^{\msf{V}}$ if the following statement is \emph{false}.
   \begin{quote}      
      (*) There exists a neighborhood, $\clU$, of $x$ and $V \in \msf{V}$ s.t.\ for every 
      $x' \in \clU \cap \D'$ we have $supp \, \Phi(x') \subset V$. 
   \end{quote}
Here ``$supp \, \Phi(x')$" is the support of the measure $\Phi(x)$. (See Federer \cite[p.\ 60]{hF69} for the definition of the support of a measure.) Let $\Ss^{\msf{V}}$ denote the set of all $\msf{V}$-severe singularities of $\Phi$. As in \eqref{E:SV.is.closed}, $\Ss^{\msf{V}}$ is closed. This notion generalizes the notion of $\msf{V}$-severe singularities defined above (definition \ref{D:V-severe}). 

Now suppose $\msf{V}$ has a convex combination function that extends to integration w.r.t.\ probability measures. I.e., if $P \in \mathbb{P}$ and $supp \, P \subset V \in \msf{V}$, 
then $\int_{\F} x \, P(dx) \in V$ with other properties generalizing those listed in definition \ref{D:convex.combo.fn}. I conjecture that theorem \ref{T:if.lin.combo.on.F.then.can.rstrct.to.bad.sings} extends to stochastic data maps to show that there exists an ordinary, i.e., deterministic, data map, $\Omega$, whose singular set lies in $\Ss^{\msf{V}}$ and having analogues of the other properties given in the theorem. This would allow some of our deterministic theory to apply to stochastic data maps. 

Criterion (*) might be weakened as follows. Let $\epsilon > 0$ be small. 
   \begin{quote}   
      There exists a neighborhood, $\clU$ of $x$, $V \in \msf{V}$, and a closed subset 
      $\msf{C} \subset V$ s.t.\ 
      for every $x' \in \clU \cap \D'$ the probability measure $\Phi(x')$ gives probability less 
      than $\epsilon$ to $\F \setminus \msf{C}$.
   \end{quote}

\section{Existence of convex combination functions}  \label{SS:existence.of.convx.combos}
Now we turn to the problem of the existence of convex combination functions. We show that if the feature space $\F$ is a smooth manifold it has an open cover with a convex combination function. 
By Boothby Boothby \cite[Theorem (4.5), p.\ 193]{wmB75}, a smooth manifold can be endowed with a Riemannian metric. Consider that done. 

We discuss two cases. In the first we show the existence of a cover with a commutative convex combination function, but do not specify how to recognize such a cover. In the second case we show that any cover consisting of geodesically convex sets (definition \ref{D:geodesic.convexity}) has at least a non-commutative convex combination function.

\subsection{Fr\'{e}chet mean}
      \begin{remark} \label{R:construct.commut.convex.combo}
 Before introducing the Fr\'{e}chet mean we sketch another idea for construction of a commutative convex combination function on a manifold. Bhattacharya and Patrangenaru \cite{rBvP03.MeansOnManifs} call it the ``extrinsic mean''. A special case of this method is described in section  \ref{S:convex.combos.on.spheres}.
    
Let $M$ be a differentiable manifold. Imbed it in $\RR^{N}$ for some $N$. Let $T$ be a tubular neighborhood about $M$ in $\RR^{N}$ (subsection \ref{SSS:tubular.nbhd}). Suppose that each point in $M$ has a neighborhood $V$ with the property that the convex hull, $Conv(V)$, of $V$ in $\RR^{N}$ lies in $T$ and that if $y \in Conv(V)$ then the unique closest point of $y$ to $M$ lies in $V$. (See proposition \ref{P:tubular.nbhd.thm}. Might any differentiable manifold have this property for some $N$?) Define a convex combination function, $\gamma( V, \cdot, \cdot)$, in $V$ as follows. If $x_{0}, ..., x_{n} \in V$ and $(\lambda_{0}, ..., \lambda_{n})$ lies in the $n$-simplex, $\Delta_{n}$, let $y$ be the convex combination 
$y := \lambda_{0} x_{0} + \lambda_{1} x_{1} + ... + \lambda_{n} x_{n} \in Conv(V) \subset T \subset \RR^{N}$ and let 
$\gamma \bigl( V, (\lambda_{0}, ..., \lambda_{n}), (x_{0}, ..., x_{n})  \bigr) \in M$ be the closest point, $z$, in $M$ to $y$, so $z \in V$. Then it seems that $\gamma$ is a commutative convex combination function on $M$ as defined in definition \ref{D:convex.combo.fn}.
      \end{remark}
      
Now we describe the ``Fr\'{e}chet'' or ``intrinsic mean'' (Bhattacharya and Patrangenaru \cite{rBvP03.MeansOnManifs}) as a commutative convex combination function on a smooth manifold. Let $\F$ be a Riemannian manifold. 
Let $\rho$ be the topological metric on $\F$ corresponding to the Riemannian metric. 
Let $x_{0} \in \F$ be arbitrary and, for $r > 0$, 
let $\mcl{B}_{r}(x_{0})$ be the open geodesic ball, 
$\bigl\{ x \in \F : \rho(x, x_{0}) < r \bigr\}$ as in \eqref{E:ball.defn} and let 
$\overline{\mcl{B}_{r}(x_{0})} \subset \F$ be its closure. 
Let $\mbf{P}_{\F}$ denote the space of probability measures on the Borel 
$\sigma$-field on $\F$. Put on $\mbf{P}_{\F}$ the topology of weak convergence 
(the weakest topology on $\mbf{P}_{\F}$ that makes each map 
$P \mapsto \int f \, dP$ continuous for every bounded continuous $f : \F \to \RR$; Billingsley \cite[p.\ 236]{pB68.ConvProbMeas}). 
Let $\mbf{P}_{\overline{x_{0}, r}} \subset \mbf{P}_{\F}$ denote the set of those elements 
of $\mbf{P}_{\F}$ s.t.\ $P \bigl( \overline{\mcl{B}_{r}(x_{0})}^{c} \bigr) = 0$, 
where $\overline{\mcl{B}_{r}(x_{0})}^{c}$ is the complement 
of $\overline{\mcl{B}_{r}(x_{0})}$ in $\F$. 

Observe that, reducing the radius $r$ if necessary, any bounded continuous real function 
on $\overline{\mcl{B}_{r}(x_{0})}$ can be extended to a bounded and continuous real function 
on $\F$  
(in such a way that the operation of extending functions 
on $\overline{\mcl{B}_{r}(x_{0})}$ is Lipschitz in the sup norm)\footnote{Take $r > 0$ sufficiently small that $\overline{\mcl{B}_{r}(x_{0})}$ lies in a normal neighborhood, $\clU$, of $x_{0}$ (Boothby \cite[p.\ 335]{wmB75}). 
Suppose $\dim \F = m$. Then $\clU$ can be parametrized by $\psi : V \to \clU$, where $V$ is a star-shaped neighborhood of 0 in $\RR^{m}$, $\psi(0) = x_{0}$, 
and $\psi \Bigl( \overline{B_{r}^{m}(0)} \Bigr) = \overline{\mcl{B}_{r}(x_{0})}$. (See \eqref{E:Euc.ball.defn}.)  
Let $f : \overline{\mcl{B}_{r}(x_{0})} \to \RR$ be bounded and continuous. Define 
$g : \overline{B_{r}^{m}(0)} \to \RR$ by $g = f \circ \psi$. Extend $g$ outside 
$\overline{B_{r}^{m}(0)}$ radially: If $y \in V \setminus \overline{B_{r}^{m}(0)}$, define $g(y) := g \bigl( r |y|^{-1} y \bigr)$. Let $\omega : \F \to [0,1]$ be a continuous function 
s.t.\ 
$\omega = 1$ on $\overline{\mcl{B}_{r}(x_{0})}$ and $\omega = 0$ outside $\clU$. Finally, extend $f$ to all of $\F$ by: 
    \begin{equation*}
      f(x) := 
        \begin{cases}
          0 , &\text{ if } x \in \F \setminus \clU, \\
          \omega(x) \, g \circ \psi^{-1}(x), &\text{ otherwise.}
        \end{cases} 
    \end{equation*}
Let $\epsilon > 0$. If $f' : \overline{\mcl{B}_{r}(x_{0})} \to \RR$ is bounded and continuous and $|f' - f| < \epsilon$ on $\overline{\mcl{B}_{r}(x_{0})}$ then the extensions satisfy the same inequality on $\F$.}. 
It follows that the weak topology on $\mbf{P}_{\overline{x_{0}, r}}$ is the same as the relative topology it inherits as a subset of $\mbf{P}_{\F}$ with its weak topology.

Let $P \in \mbf{P}_{\overline{x_{0}, r}}$ and define
    \begin{equation}  \label{E:gP.defn}
      g_{P}(y) := g(y,P) := \int_{\overline{\mcl{B}_{r}(x_{0})}} \rho^{2}(y, z) \, P(dz) , 
        \qquad y \in \overline{\mcl{B}_{r}(x_{0})} .
    \end{equation}
We show in the proof of proposition \ref{P:smooth.manifs.have.commutative.convex.combos} that if $r > 0$ is sufficiently small, then $g_{P}$ has a unique minimum. That idea can be used to define a cover of $\F$ which has a convex combination function, viz., 
the Fr\'{e}chet, intrinsic, or Karcher mean. Actually, $\F$ only needs to be homeomorphic to a Riemannian manifold. Look after \eqref{E:infinite.convex.combo} for the definition of a commutative convex combination function.

    \begin{prop}  \label{P:smooth.manifs.have.commutative.convex.combos}
Let $\msf{M}$ be a Riemannian manifold and let $h : \F \to \msf{M}$ be a homeomorphism. Then there exists an open cover
$\msf{V}$ of $\F$, consisting of inverse images (under $h$) of certain open geodesic balls in $\msf{M}$ s.t.\ there is a commutative convex combination function 
on $\msf{V}$. 
   \end{prop}

\begin{proof} WLOG $\F = \msf{M}$. Let $x_{0} \in \F$ be arbitrary. \emph{Claim:} 
    \begin{equation}  \label{E:g(y,P).is.cont}
      g \text{ is continuous on } 
        \mcl{B}_{r}(x_{0}) \times \mbf{P}_{\overline{x_{0}, r}}.
    \end{equation} 
Notice that if
$y_{1}, y_{2}, z \in \overline{\mcl{B}_{r(x_{0}}(x_{0})}$ then
    \begin{align}   \label{E:rho.y1.rho.y2.close}
      \bigl| \rho^{2}(y_{1}, z) - \rho^{2}(y_{2}, z) \bigr| 
       &= \bigl| \rho(y_{1}, z) - \rho(y_{2}, z) \bigr| \bigl| \rho(y_{1}, z) + \rho(y_{2}, z) \bigr| 
         \notag \\
       &\leq 4 r \bigl| \rho(y_{1}, z) - \rho(y_{2}, z) \bigr|| \\
       &\leq 4 r \rho(y_{1}, y_{2}) .  \notag
    \end{align}

Let $\epsilon > 0$ be given. Pick $y_{1}, y_{2} \in \overline{\mcl{B}_{r(x_{0}}(x_{0})}$ s.t.\ 
$\rho(y_{1}, y_{2}) < \epsilon/(8 r)$. Notice that  the function 
$z \mapsto \rho^{2}(y_{1}, z)$ is bounded and continuous on $\mcl{B}_{r(x_{0}}(x_{0})$. 
Pick $P_{1}, P_{2} \in \mbf{P}_{\overline{x_{0}, r}}$ so close w.r.t.\ the topology on $\mbf{P}_{\F}$ that
    \begin{equation*}  \label{E:P1.P2.integrals.close}
      \left| \int_{\mcl{B}_{r}(x_{0})} \rho^{2}(y_{1} ,z) \, P_{1}(dz)  
        -  \int_{\mcl{B}_{r}(x_{0})} \rho^{2}(y_{1}, z) \, P_{2}(dz) \right| < \epsilon/2 .
    \end{equation*}
Then, by \eqref{E:rho.y1.rho.y2.close},
    \begin{align*}
      \left| \int_{\mcl{B}_{r}(x_{0})}  \rho^{2}(y_{1}, z) \, P_{1}(dz) \right.
        &-  \left. \int_{\mcl{B}_{r}(x_{0})}  \rho^{2}(y_{2}, z) \, P_{2}(dz) \right| \\
        &\leq \left| \int_{\mcl{B}_{r}(x_{0})}  \rho^{2}(y_{1}, z) \, P_{1}(dz)  
          -  \int_{\mcl{B}_{r}(x_{0})} \rho^{2}(y_{1}, z) \, P_{2}(dz) \right| \\
        & \qquad \qquad + \int_{\mcl{B}_{r}(x_{0})}  
              \bigl| \rho^{2}(y_{1}, z)  
                -  \rho^{2}(y_{2}, z) \bigr| \, P_{2}(dz) \\
        &< \epsilon .
    \end{align*}
This proves the claim \eqref{E:g(y,P).is.cont}.

Denote by $\mbf{P}_{x_{0}, r} \subset \mbf{P}_{\overline{x_{0}, r}}$ the set of all probability measures, $P \in \mbf{P}_{\overline{x_{0}, r}}$, on $\F$ s.t.\ 
$P \bigl( \mcl{B}_{r}(x_{0})^{c} \bigr) = 0$ Karcher \cite[Theorem 1.2, p.\ 510]{hK77.RiemannianCenterOfMass} asserts that:
    \begin{multline}  \label{E:properties.of.r(x0)}
          \text{There exists $r(x_{0}) > 0$ s.t., } \mcl{B}_{r(x_{0})}(x_{0}) 
            \text{ is geodesically convex } \\
          \text{ and if $r \in \bigl( 0, r(x_{0}) \bigr]$ and $P \in \mbf{P}_{x_{0}, r}$, then } \\
        \text{ there is a unique }
        y = \alpha(P; x_{0}, r) \in \mcl{B}_{r}(x_{0}) \\
         \text{ at which $g_{P}$ achieves its minimum \emph{on} } 
          \overline{\mcl{B}_{r(x_{0})/2}(x_{0})} .
    \end{multline} 
Thus, the minimum in the interior, $\mcl{B}_{r}(x_{0})$ is actually the unique minimum on the closure, $\overline{\mcl{B}_{r}(x_{0})}$. (This requires that the sectional curvature of $\F$ in $\mcl{B}_{r}(x_{0})$ be bounded above. 
But if $r(x_{0}) > 0$ is sufficiently small it will be. See Bhattacharya and Patrangenaru \cite[Remark 2.1, pp.\ 6--7]{rBvP03.MeansOnManifs}. A careful argument is given in \cite[Section 2.1]{spE2022.ReadingKarcher77}. \cite{spE2022.ReadingKarcher77} is a careful reading of much of Karcher \cite{hK77.RiemannianCenterOfMass}.) 
In fact, Karcher's result tells us more, viz., $g_{P}$ is strictly convex along geodesics 
in $\mcl{B}_{r(x_{0})}(x_{0})$.

Observe that $\mbf{P}_{x_{0}, r}$ is convex, in the usual sense of the term: 
$\lambda \in [0,1]$, $P_{0}, P_{1} \in \mbf{P}_{x_{0}, r}$ imply 
$(1-\lambda) P_{0} + \lambda P_{1} \in \mbf{P}_{x_{0}, r}$. Hence, \emph{a fortiori} it is path-wise connected. A Riemannian manifold, in particular $\mcl{B}_{r}(x_{0})$, is a separable metric space, therefore normal. Hence, by 
Billingsley \cite[pp.\ 238--239]{pB68.ConvProbMeas}, 
$\mbf{P}_{x_{0}, r}$ is also a separable metric space. Therefore, \eqref{E:D.metric.F.normal} holds with $ \mbf{P}_{x_{0}, r}$ playing the role 
of $\D$ and $\mcl{B}_{r}(x_{0})$ playing the role of $\F$. 
Take $r(x_{0}) \geq r > 0$ so small that 
$\overline{\mcl{B}_{r}(x_{0})}$ is compact and $\mcl{B}_{r}(x_{0})$ is geodesically convex (definition \ref{D:geodesic.convexity}). (This is possible by proposition \ref{P:geod.cnvx.nbhds.exist}.) Recalling \eqref{E:g(y,P).is.cont}, we may 
apply lemma \ref{L:data.maps.defined.by.opt}\eqref{I:cmpct.opt} 
with $\D' = \D = \mbf{P}_{x_{0}, r}$ and $\F = \overline{\mcl{B}_{r}(x_{0})}$  
we see that 
    \begin{equation*}
      \alpha(\cdot; x_{0}, r) \text{ is continuous on } \mbf{P}_{x_{0}, r} .
    \end{equation*} 

In definition \eqref{E:gP.defn} the ball $\mcl{B}_{r}(x_{0})$ appears explicitly. 
But it extends  
to all of $\F$:
    \begin{equation}  \label{E:gP.on.F.defn}
      g_{P}(y) = \int_{\F} \rho^{2}(y, z) \, P(dz) , \qquad y \in \F .
    \end{equation}
So $g_{P}(y)$ might be infinite. Now $g_{P}$ only depends on $\mcl{B}_{r}(x_{0})$ implicitly through $P$. Suppose $P \in \mbf{P}_{x_{0}, r(x_{0})/2}$. We show that $g_{P}$ achieves a minimum value in $\F$ at only one point  and that point is in 
$V :=  V(x_{0}) := \mcl{B}_{r(x_{0})/2}(x_{0})$. Trivially, as observed in 
Bhattacharya and Patrangenaru \cite[Remark 2.1, pp.\ 6--7]{rBvP03.MeansOnManifs} again, we have $\min_{y \in V} g_{P}(y) \leq g_{P}(x_{0}) < r(x_{0})^2/4$, but if 
$y \in \F \setminus \mcl{B}_{r(x_{0})}(x_{0})$, i.e. $y$ lies outside the ball with twice the radius, $g_{P}(y) \geq r(x_{0})^2/4$, because all of the mass of $P$ lies inside 
$\mcl{B}_{r(x_{0}/2)}(x_{0})$. Thus, $g_{P}$ achieves its minimum (over $\F$) in the larger ball $\mcl{B}_{r(x_{0})}(x_{0})$. By \eqref{E:properties.of.r(x0)}, $g_{P}$ achieves its minimum over $\overline{\mcl{B}_{r(x_{0})}(x_{0})}$ at just one point, $y_{1}$, in the interior, $\mcl{B}_{r(x_{0})}(x_{0})$. Similarly, $g_{P}$ achieves its minimum over 
$\overline{\mcl{B}_{r(x_{0})/2}(x_{0})}$ at just one point, $y_{0}$, 
in $\mcl{B}_{r(x_{0})/2}(x_{0})$. 

Suppose $y_{0} \neq y_{1}$. 
Since $\overline{\mcl{B}_{r(x_{0})/2}(x_{0})} \subset \mcl{B}_{r(x_{0})}(x_{0})$, we must have
     \begin{equation*}
       g_{P}(y_{0}) > g_{P}(y_{1}) .
    \end{equation*} 
By \eqref{E:properties.of.r(x0)}, $\mcl{B}_{r(x_{0})}(x_{0})$ is geodesically convex. Let $c$ be a shortest geodesically joining $y_{1}$ to $y_{0}$. Say, $c(0) = y_{0}$ 
and $c(1) = y_{1}$. By definition \ref{D:geodesic.convexity}, $c$ lies entirely 
in $\mcl{B}_{r(x_{0})}(x_{0})$. For some $t \in (0,1]$, $c(t)$ is in the boundary 
of $\mcl{B}_{r(x_{0})/2}(x_{0})$. I.e., 
$c(t) \in \overline{\mcl{B}_{r(x_{0})/2}(x_{0})} \setminus \mcl{B}_{r(x_{0})/2}(x_{0})$, which means $c(t) \neq y_{0} \in \mcl{B}_{r(x_{0})/2}(x_{0})$. But we have already recognized that $g_{P}$ is strictly convex along geodesics in $\mcl{B}_{r(x_{0})}(x_{0})$, in particular along $c$. Thus, since $g_{P}(y_{0}) > g_{P}(y_{1})$ and $t > 0$, 
    \begin{equation*} 
      g_{P}(y_{0}) > (1-t) g_{P}(y_{0}) + t g_{P}(y_{1}) \geq g_{P} \bigl[ c(t) \bigr].
    \end{equation*} 
But this contradicts the fact that $g_{P}$ achieves its minimum over 
$\overline{\mcl{B}_{r(x_{0})/2}(x_{0})}$ only 
at $y_{0} \in \mcl{B}_{r(x_{0})/2}(x_{0})$. We conclude $y_{0} = y_{1}$. 
I.e., $g_{P}$ achieves its minimum over $\F$ only at 
$y_{0} \in V(x_{0}) := \mcl{B}_{r(x_{0})/2}(x_{0})$. 
Thus, $\alpha(P; x_{0}, r(x_{0})/2))$ does not depend on $x_{0}$ or $r(x_{0})$. We therefore write 
$\alpha(P) := \alpha(P; x_{0},r(x_{0})/2) \bigr)$. Thus, for $P \in \mbf{P}_{x_{0}, r(x_{0})/2}$, $\alpha(P)$ is the minimizer of $g_{P}$ over $\F$. It is well-defined and continuous in $P \in \mbf{P}_{x_{0}, r(x_{0})/2}$. 
 
Let $\msf{V} := \bigl\{ V(x_{0}) \subset \F : x_{0} \in \F \bigr\}$. 
Let $k = 0, 1, 2, \ldots$. Define $\Delta_{k}$ as in \eqref{E:Delta.k.defn}. Define a map 
$m_{k} : \Delta_{k} \times \F^{k+1} \to \mbf{P}_{\F}$ by
    \begin{multline*}
      m_{k} \bigl[ (\lambda_{0}, \lambda_{1}, \ldots, \lambda_{k}), 
        (y_{0}, y_{i}, \ldots, y_{k}) \bigr] 
          := \sum_{j=0}^{k} \lambda_{j} \delta_{y_{j}} \in \mbf{P}_{\F} , \\
            (\lambda_{0}, \lambda_{1}, \ldots, \lambda_{k}) \in \Delta_{k} , \;
              y_{0}, y_{i}, \ldots, y_{k} \in \F , 
    \end{multline*}
where $\delta_{x}$ denotes unit mass at $x$. Then 
$m_{k} : \Delta_{k} \times \F^{k+1} \to \mbf{P}_{\F}$ is continuous w.r.t.\ the product topology on $\Delta_{k} \times \F^{k+1}$ and the weak topology 
on $\mbf{P}_{\F}$.\footnote{Let 
$\blds{\lambda} := (\lambda_{0}, \lambda_{1}, \ldots, \lambda_{k}) \in \Delta_{k}$ and
$y_{0}, y_{i}, \ldots, y_{k} \in \F$. Let $f : \F \to \RR$ be bounded and continuous. 
Pick $C < \infty$ s.t.\ $|f| < C$. Let $\epsilon > 0$. 
Pick $\delta > 0$ s.t.\ $z \in \F$ with $\rho(y_{i},z) < \delta$ implies 
$\bigl| f(y_{i}) - f(z) \bigr| < \epsilon/2$ ($i = 0, \ldots, k$).
Let $x_{0}, x_{i}, \ldots, x_{k} \in \F$ with $\rho(y_{i}, x_{i}) < \delta$ 
($i=0, 1, \ldots, k$). Write $\mbf{y} := (y_{0}, y_{i}, \ldots, y_{k})$ and
$\mbf{x} := (x_{0}, x_{i}, \ldots, x_{k})$. 
Let $\blds{\mu} = (\mu_{0}, \mu_{1}, \ldots, \mu_{k}) \in \Delta_{k}$ satisfy 
$\sum_{i=0}^{k} |\mu_{i} - \lambda_{i}| < \epsilon/\bigl[ 2(k+1)C \bigr]$. Then  
    \begin{multline*}
      \left| \int_{\F} f(z) \, m_{k}(\blds{\lambda}, \mbf{x})(dz) 
        - \int_{\F} f(z) \, m_{k}(\blds{\mu}, \mbf{y})(dz) \right|
          \leq \sum_{i=0}^{k} \bigl| \lambda_{i} f(y_{i}) - \mu_{i} f(x_{i}) \bigr| \\
          \leq \sum_{i=0}^{k} | \lambda_{i} - \mu_{i} | \bigl| f(x_{i} \bigr|
          + \sum_{i=0}^{k} \lambda_{i} \bigl| f(y_{i}) - f(x_{i}) \bigr|
           < \epsilon . 
    \end{multline*}
  } 
$m_{k}$ is also commutative in the sense that $m_{k}(\blds{\lambda}, \mbf{x})$ is unchanged when the components of $\blds{\lambda} \in \Delta_{k}$ and $\mbf{x} \in \F^{k+1}$ undergo the same arbitrary permutation. 

Let $\blds{\lambda} \in \Delta_{k}$. If $\mbf{y} \in V^{k+1}$, where $V \in \msf{V}$, the support of $m_{k}(\blds{\lambda}, \mbf{y})$ lies in $V$. Hence, 
$\alpha \circ m_{k}(\blds{\lambda}, \mbf{y})$ is defined and is the composition of two continuous functions, hence is continuous. Recall the definition, \ref{D:convex.combo.fn}, of convex combination functions. Let $\gamma_{k}$ denote the map 
    \begin{multline*}
      \gamma_{k} : \bigcup_{V \in \msf{V}} \bigl( \{ V \} \times \Delta_{k} \times V^{k+1} \bigr) 
        \to \F \text{ given by }  
          \gamma_{k} (V, \blds{\lambda}, \mbf{x})
            = \alpha \circ m_{k} (\blds{\lambda}, \mbf{x} ) , \\
              \qquad V \in \msf{V}, \; \blds{\lambda} \in \Delta_{k} , \; \mbf{x} \in V^{k+1} .
    \end{multline*}

It is immediate that $\gamma_{k}$ has properties \ref{I:convex.combos.are.local} through \ref{I:1.is.identity} in definition \ref{D:convex.combo.fn}. In particular, we may drop the first argument $V$. Combining $\gamma_{k}$ over $k = 0, 1, \ldots$ into a function 
$\gamma :  \bigcup_{k=0}^{\infty} \bigcup_{V \in \msf{V}} \; \Delta_{k} \times V^{k+1} \to \F$, we see that $\gamma$ satisfies \eqref{E:drop.0.coefs} because the analogous equation holds 
for $m_{k}$ ($k=0, 1, \ldots$). $\gamma$ also inherits the commutativity of the $m_{k}$'s.
  \end{proof}

\subsection{Convex combination by recursion} \label{SS:convex.combos.by.recursion}
Proposition \ref{P:smooth.manifs.have.commutative.convex.combos} tells us that if $\F$ is homeomorphic to a smooth manifold then \emph{there exists} an open cover, $\msf{V}$ 
of $\F$ on which there is a commutative convex combination function. The proposition does not tell us how to find such a cover. Here we show (proposition \ref{P:smooth.manifs.have.convex.combos}) that any cover consisting of geodesically convex neighborhoods (definition \ref{D:geodesic.convexity}) has a convex combination function, but it may not be commutative.

Suppose that $\gamma$ is a convex combination function on a covering $\msf{V}$ 
of $\F$. Let $V \in \msf{V}$ and let $x_{1}, \, x_{2} \in V$. Define
	\[
            \varphi_{x_{1} x_{2}} (\lambda) 
              := \gamma \bigl[ ( \lambda, 1-\lambda ), ( x_{2}, x_{1} )\bigr]	
                \quad \lambda \in [0,1].
	\]
Then $\varphi_{x_{1} x_{2}} \bigl( [0,1] \bigr)$ is a curve joining $x_{1}$ and $x_{2}$ that lies
entirely inside $V$. Moreover $\varphi_{x_{1} x_{2}} (\lambda)$ is continuous in 
$( \lambda, x_{1}, x_{2} )$. Thus, convex combination functions define certain families of
curves. It turns out that the converse is also true. We prove the following below.
	\begin{prop}  \label{P:make.conv.combo.fns.from.curves}  
	Suppose there is an open subset $X \subset \F \times \F$ 
s.t.\ $\bigl\{ (x,x) : x \in \F \bigr\} \subset X$ and having the property that if $(x_{1}, x_{2}) \in X$ then there is a finite number $s_{x_{1}, x_{2}} \geq 0$ and a curve $\varphi_{x_{1}, x_{2}} : I_{x_{1}, x_{2}} \to \F$ joining $x_{1}$ and $x_{2}$, where 
$I_{x_{1}, x_{2}} = [0, s_{x_{1}, x_{2}}]$. Thus, $\varphi_{x_{1}, x_{2}}(0) = x_{1}$ 
and $\varphi_{x_{1}, x_{2}}(s_{x_{1}, x_{2}}) = x_{2}$. Suppose the function $(x_{1}, x_{2}) \to s_{x_{1}, x_{2}}$ is continuous and $s_{x_{1}, x_{2}} = 0$ if and only if $x_{1} = x_{2}$. Define
	\begin{equation*}
		E := \bigl\{ (s, x_{1}, x_{2}) \in \RR \times X  : s \in I_{x_{1}, x_{2}} \bigr\}.
	\end{equation*}
 Give $E$ the relative topology it inherits as a subset of $\RR \times \F \times \F$. Suppose the curves $\varphi_{x_{1}, x_{2}}$ have the property that the function
		\[
		    \Gamma: (s; x_{1}, x_{2}) \mapsto \varphi_{x_{1}, x_{2}}(s) \in  \F, 
			    \quad (s, x_{1}, x_{2}) \in E 
		\]
	is continuous on $E$. Let $\msf{V}$ be an open covering of $\F$  and suppose that for every
	$V \in \msf{V}$ the product $V \times V \subset X$ and if  $x_{1}, x_{2} \in V$ then 
	$\varphi_{x_{1}, x_{2}}  [ I_{x_{1}, x_{2}} ] \subset V$. Then there is a convex combination function on $\msf{V}$. That convex combination function may not be commutative. 
	\end{prop}

We prove the preceding presently. First, we use it to prove the following. Recall the definition,  \ref{D:geodesic.convexity}, of geodesically convex subset of a Riemannian manifold. 
By Boothby \cite[Theorem 4.5, p.\ 193]{wmB75} any smooth manifold can be equipped with a Riemannian tensor. 
    \begin{prop}  \label{P:smooth.manifs.have.convex.combos}
If $\F$ is a Riemannian manifold then \emph{any} open cover, $\msf{V}$ of $\F$ consisting of geodesically convex subsets of $\F$ has an, at least noncommutative, convex combination function. 
   \end{prop}

\begin{proof}[Proof of proposition \ref{P:smooth.manifs.have.convex.combos}]
Let $\phi$  be the topological metric on $\F$ induced by the Riemannian tensor. We use proposition \ref{P:geod.cnvx.nbhds.exist}.

Let $\epsilon \in (0,1)$ and for every $x_{0} \in \F$ let $r(x_{0})$ be as in proposition \ref{P:geod.cnvx.nbhds.exist} and let $V(x_{0})$ be the open ball
	\[
		V(x_{0}) := \bigl\{ y \in \F : \phi(y, x_{0}) < (1-\epsilon) r(x_{0}) \bigr\}.   
	\]
Let $\msf{V} = \bigl\{ V(x) : x \in \F \bigr\}$. By making $r(x_{0})$ smaller if necessary, we may assume $\overline{V(x_{0})}$ is compact. It follows from proposition \ref{P:geod.cnvx.nbhds.exist} that
   \begin{equation}   \label{E:V.bar.is.geod.convx}
      \text{$\overline{V(x_{0})}$ is geodesically convex for every $x_{0} \in \F$.}
   \end{equation}

Let 
	\begin{equation*}
		X := \bigl\{ (x_{1}, x_{2}) \in \F \times \F : 
		  \text{There exists } x \in \F \text{ s.t.\ } x_{1}, x_{2} \in V(x) \bigr\}.
	\end{equation*}
Clearly, $X$ is an open subset of $\F \times \F$ and obviously if $x \in \F$ then $V(x) \times V(x) \subset X$. If $(x_{1}, x_{2}) \in X$, let $s_{x_{1}, x_{2}} := \phi(x_{1}, x_{2})$. Obviously, $s_{x_{1}, x_{2}}$ is continuous in $(x_{1}, x_{2}) \in X$ and $s_{x_{1}, x_{2}} = 0$ if and only if $x_{1} = x_{2}$.

Let $x_{0} \in \F$ and let $x_{1}, x_{2} \in  V(x_{0})$. 
Let $\varphi_{x_{1}x_{2}} : [0, s_{x_{1}, x_{2}}] \to \F$ be the unique shortest geodesic in $\F$ joining $x_{1}$ and $x_{2}$. (Uniqueness by part \ref{I:geodesically.convex} of proposition \ref{P:geod.cnvx.nbhds.exist}. 
$\varphi_{x_{1} x_{2}}$ is parametrized by arclength; 
Boothby \cite[lemma 5.2, p.\ 327]{wmB75}.)
Notice, that $\varphi$ is actually defined on an open interval containing $[0, s_{x_{1}, x_{2}}]$ 
(Boothby \cite[Corollary (5.6), pp.\ 329 --330]{wmB75}). Then by proposition 
\ref{P:geod.cnvx.nbhds.exist}(\ref{I:geodesically.convex}), we have
$\varphi_{x_{1}x_{2}} [0, s_{x_{1}, x_{2}}] \subset V(x_{0})$.    
If  $x_{1}, x_{2} \in V(x_{0})$ and $0 \leq s \leq s_{x_{1}, x_{2}}$, let 
   \begin{equation}  \label{E:Gamma.doesn't.leave.V}
     \Gamma(s; x_{1}, x_{2}) = \varphi_{x_{1} x_{2}}(s) \in  
            V(x_{0}).
   \end{equation}
In particular, 
   \begin{multline}  \label{E:trivial.properties.of.Gamma}
      \Gamma(0; x_{1}, x_{2}) = x_{1},  \;\Gamma(s_{x_{1}, x_{2}}; x_{1}, x_{2}) = x_{2}, \text{ and } 
        \; \Gamma(0 ; x, x) = x  \\ 
           \text{ for } s \in [0,s_{x_{1}, x_{2}}],  
             \; x_{1}, x_{2}, x \in V(x_{0}).
   \end{multline}

\emph{Claim:}  $\Gamma: (s; x_{1}, x_{2}) \mapsto \varphi_{x_{1}, x_{2}}(s)$ is continuous in its three arguments (on $E$ as defined in proposition \ref{P:make.conv.combo.fns.from.curves}). Let
$x_{1}, x_{2}, x_{1}', x_{2}' \in V(x_{0})$, $s \in [0, s_{x_{1}, x_{2}}]$, and $s' \in [0, s_{x_{1}', x_{2}'}]$. 
Imagine $x_{i}' \to x_{i}$ ($i=1,2$) and $s' \to s$. (We know that $s_{x_{1}', x_{2}'} \to s_{x_{1}, x_{2}}$ 
as $x_{i}' \to x_{i}$ ($i=1,2$) so $s' \to s$ is possible.) Suppose $\Gamma(s'; x_{1}', x_{2}')$ does not converge to 
$\Gamma(s; x_{1}, x_{2})$ Then, by compactness of $\overline{V(x_{0})}$,  there is a 
point $x \in \overline{V(x_{0})} \setminus \bigl\{ \Gamma(s; x_{1}, x_{2}) \bigr\}$, and there are sequences $\{ s_{m}' \} \subset [0, \infty)$, $x_{1m}' \to x_{1}$ and $x_{2m}' \to x_{2}$ s.t.\ 
$s_{m}' \in [0,  s_{x_{1m}', x_{2m}'}]$ for every $m$, $s_{m}' \to s$, and 
$\Gamma(s_{m}' ; x_{1m}', x_{2m}')$ does not converge to $\Gamma(s; x_{1}, x_{2})$ but instead
$\Gamma(s_{m}' ; x_{1m}', x_{2m}') \to x$ as $m \to \infty$. First, notice that 
   \begin{multline*}
      \phi(x_{1}, x) 
        = \lim_{m \to \infty} \phi \bigl( x_{1m}', \, \Gamma(s_{m}' ; x_{1m}', x_{2m}') \bigr)   
          = \lim_{m \to \infty} \phi \bigl(  x_{1m}',  \,\varphi_{x_{1m}', x_{2m}'} (s_{m}') \bigr) \\
            = \lim_{m \to \infty} s_{m}' =  s 
              = \phi \bigl(  x_{1}, \, \varphi_{x_{1}, x_{2}}(s) \bigr) 
                = \phi \bigl( x_{1}, \, \Gamma(s; x_{1}, \, x_{2}) \bigr). 
   \end{multline*}
Suppose $x$ were a point on the geodesic joining $x_{1}$ to $x_{2}$, 
i.e., on $\varphi_{x_{1}, x_{2}}$,. Then, by the preceding, we must have $x =  \varphi_{x_{1}, x_{2}}(s)$, i.e., $\lim_{m \to \infty} \Gamma(s_{m}' ; x_{1m}', x_{2m}') = x = \Gamma(s ; x_{1}, x_{2})$, contradicting the assumption that $\Gamma(s_{m}' ; x_{1}', x_{2}')$ does not converge to 
$\Gamma(s; x_{1}, x_{2})$. Now, since $\Gamma(s_{m}'; x_{1m}', x_{2m}')$ lies on the shortest geodesic arc joining $x_{1m}'$ and $x_{2m}'$, we have
    \begin{equation}  \label{E:phi.x1m'.Gamma.sum}
 	  \phi (x_{1m}', \, \Gamma(s_{m}'; x_{1m}', x_{2m}')) 
	     + \phi (\Gamma(s_{m}'; \, x_{1m}', x_{2m}'), \, x_{2m}') = \phi (x_{1m}', x_{2m}').
    \end{equation}
    
Let $\beta$ be the piecewise $C^{1}$ curve that goes from $x_{1}$ to $x$ along a shortest geodesic and then along another shortest geodesic from $x$  to $x_{2}$. It now follows from \eqref{E:phi.x1m'.Gamma.sum} that the length of $\beta$ is 
   \begin{multline*}
      \phi (x_{1}, x ) + \phi (x , x_{2}) \leq \phi (x_{1}, x_{1m}') 
         + \phi (x_{1m}', \, \Gamma(s_{m}'; x_{1m}', x_{2m}'))  \\
      + \phi (\Gamma(s_{m}'; x_{1m}', x_{2m}'), \, x ) 
      + \phi (x , \, \Gamma(s_{m}'; x_{1m}', x_{2m}')) 
         + \phi (\Gamma(s_{m}'; x_{1m}', x_{2m}'), \, x_{2m}') + \phi (x_{2m}', x_{2})  \\
      = \phi (x_{1},x_{1m}') + 2 \phi (x , \, \Gamma(s_{m}'; x_{1m}', x_{2m}')) + 
      \phi (x_{1m}', x_{2m}') + \phi (x_{2m}', x_{2}).
   \end{multline*}
Since $\Gamma(s_{m}'; x_{1m}', x_{2m}') \to x$, $x_{jm}' \to x_{j}$ ($j=1,2$), and $\phi$ is continuous, the expression following the equal sign in the preceding converges 
to $\phi (x_{1}, x_{2})$. 

Thus, $\beta$  is a piecewise $C^{1}$ curve in $F$ parametrized by arclength and joining $x_{1}$ to $x_{2}$ which is no longer the unique geodesic curve
$\varphi_{x_{1} x_{2}}$, in $\overline{V(x_{0})}$ joining $x_{1}$ to $x_{2}$. By \eqref{E:V.bar.is.geod.convx} above and proposition \ref{P:geod.cnvx.nbhds.exist}, this means that $\beta  = \varphi_{x_{1} x_{2}}$. 
(See also Boothby \cite[Theorem (7.2), p.\ 340]{wmB75}.) But $\beta$  passes through $x$  and $\varphi_{x_{1} x_{2}}$ does \emph{not} pass through $x$. This contradiction establishes that 
$\Gamma(s'; x_{1}', x_{2}') \to \Gamma(s; x_{1}, x_{2})$ as $s' \to s$, $x_{1}' \to x_{1}$, and $x_{2}' \to x_{2}$. 

The proposition now follows from proposition \ref{P:make.conv.combo.fns.from.curves}.
\end{proof}

\begin{proof}[Proof of propostion \ref{P:make.conv.combo.fns.from.curves}]
We will define a convex combination function, $\gamma$, on $\msf{V}$. Let $V \in \msf{V}$. Define 
    \begin{equation} \label{E:0-convex.gamma}
        \gamma(V, 1, x_{0}) = x_{0}, \quad x_{0} \in V, \; V \in \msf{V}.
    \end{equation} 
So $\gamma$ is a 0-convex combination function. 

Notice that 
   \begin{multline}  \label{E:simple.Gamma.facts}
      \Gamma(s; x_{0}, x_{1}) \in V, \\
      \Gamma(0; x_{0}, x_{1}) = \varphi_{x_{0}, x_{1}} (0) = x_{0}, \text{ and } 
        \Gamma(s_{x_{0}, x_{1}}; x_{0}, x_{1}) = \varphi_{x_{0}, x_{1}} (s_{x_{0}, x_{1}})  
          = x_{1}, \\
            \text{ for } x_{0}, x_{1} \in V, s \in I_{x_{0}, x_{0}}.
   \end{multline}
If $\lambda \in  [0,1]$ and $x_{0}, x_{1} \in V$, write 
    \begin{equation}  \label{E:1-convex.gamma}
		\gamma \bigl[ V, (1 - \lambda, \lambda), ( x_{0}, x_{1} ) \bigr]  
			:= \Gamma ( \lambda s_{x_{0}, x_{1}} ; x_{0}, x_{1} ) \in V.
    \end{equation}
Then $\gamma$ restricts to a 1-convex combination function on $\msf{V}$ and \eqref{E:drop.0.coefs} holds with $k = 1$.

Let $m = 2, 3, \ldots$ and assume inductively the following the following holds :
    \begin{multline} \label{E:gamma.induction.hyp}
      \text{For every } k = 0, \ldots, m-1 \text{ we have: For every } 
      (\mu_{0}, \ldots, \mu_{k}) \in \Delta_{k}; \; V \in \msf{V}; \\
        \text{ and } x_{0}, \ldots , x_{k} \in  V \text{ there is defined } 
            \gamma \bigl[ V,  (\mu_{0}, \ldots, \mu_{k}), (x_{0}, \ldots, x_{k}) \bigr] \in V \\
              \text{ having the properties } 
                \text{of a $k$-convex combination function on } \msf{V} \\
                \text{ for which \eqref{E:drop.0.coefs} holds.}
    \end{multline} 

If $k = 1, \ldots, m-1$ further assume also the following. Let $x_{0}, \ldots , x_{k} \in  V$. Let $(\lambda_{0}, \ldots , \lambda_{k}) \in  \Delta_{k}$. If $\lambda_{k} < 1$, write 
	\begin{equation}  \label{E:y.k-1.gamma.defn}
		y_{k-1} :=  \gamma \bigl[ V,  (1-\lambda_{k})^{-1} 
		              (\lambda_{0}, \ldots, \lambda_{k-1}), (x_{0}, \ldots, x_{k-1}) \bigr] \in V.
	  \end{equation}
Then,
   \begin{equation}  \label{E:Gamma.combo.defn}
      \gamma \bigl[ V,  (\lambda_{0}, \ldots, \lambda_{k}), (x_{0}, \ldots, x_{k}) \bigr]  =  
         \begin{cases}
	         x_{k}, &\text{ if } \lambda_{k} = 1;  \\
		         \Gamma ( \lambda_{k} s_{y_{k-1}, x_{k}} ;  y_{k-1}, x_{k} ), 
	            &\text{ if } 0 \leq \lambda_{k} < 1.
         \end{cases}
   \end{equation}
Thus, part of the induction hypothesis is that \eqref{E:Gamma.combo.defn} is consistent with definition \ref{D:convex.combo.fn}. Notice that our inductive hypotheses hold for $m = 2$. In particular, 
\eqref{E:0-convex.gamma} and \eqref{E:1-convex.gamma}  parallel  \eqref{E:y.k-1.gamma.defn} and \eqref{E:Gamma.combo.defn}, with $k=1$, and $\lambda_{k} = \lambda$.

Now use \eqref{E:Gamma.combo.defn} with $k = m$ to \emph{extend} the domain of $\gamma$ to include
$\bigcup_{V \in \msf{V}} \bigl( \{ V \} \times \Delta_{m} \times V^{m+1} \bigr)$. I.e., \emph{define} $\gamma$ for $k=m$ by \eqref{E:Gamma.combo.defn}. We prove continuity of $\gamma$ presently, but first notice that, since $y_{k-1} \in V$ (if $\lambda_{k} < 1$) and $x_{k} \in V$, we have, by \eqref{E:simple.Gamma.facts}, that definition \ref{D:convex.combo.fn}, part \ref{I:convex.combos.are.local} holds:
$\gamma \bigl[ V,  (\lambda_{0}, \ldots, \lambda_{m}), (x_{0}, \ldots, x_{m}) \bigr] \in V$. Definition \ref{D:convex.combo.fn}, part (\ref{I:consistency.of.conv.combos}) holds also: 
$\gamma \bigl[ V,  (\lambda_{0}, \ldots, \lambda_{m}), (x_{0}, \ldots, x_{m}) \bigr]$ does not depend 
on $V \supset \{ x_{0}, \ldots , x_{m} \}$ ($V \in \msf{V}$).

We prove definition \ref{D:convex.combo.fn}, part (\ref{I:gamma.k.cont}), \emph{viz.}, that 
$\gamma \bigl[ V,  (\lambda_{0}, \ldots, \lambda_{m}), (x_{0}, \ldots, x_{m}) \bigr]$ is continuous in $x_{0}, \ldots , x_{m}$ and $\lambda_{0}, \ldots , \lambda_{m}$. 

Assume we do have a failure of continuity. Then, by Ash \cite[Theorem A2.14, pp.\ 375--376]{rbA72}, for some $V \in \msf{V}$, there exists 
$\blds{\lambda}_{0} = (\lambda_{0}, \ldots, \lambda_{m})  \in \Delta_{m}$, 
$\mbf{x}_{0} = (x_{0}, \ldots, x_{m}) \in V^{m+1}$, a neighborhood $A$ 
of $\gamma \bigl( V,  \blds{\lambda}_{0}, \mbf{x}_{0} \bigr) \in V$, and a net 
(Ash \cite[Definition A2.2, p.\ 371]{rbA72}) 
$\mathfrak{N} := \bigl\{ (\blds{\lambda}_{\alpha}, \mbf{x}_{\alpha}), \alpha \in D \bigr\}$, where $D$ is a directed set,
$\blds{\lambda}_{\alpha} = (\lambda_{\alpha 0}, \ldots, \lambda_{\alpha m}) 
\in \Delta_{m}$, 
and $\mbf{x}_{\alpha} = (x _{\alpha 0}, \ldots, x _{\alpha m}) \in V^{m+1}$, s.t.\ $\mathfrak{N}$ converges to 
$(\blds{\lambda}_{0}, \mbf{x}_{0})$ (Ash \cite[Definition A2.2, p.\ 371]{rbA72} again), but for every $\beta \in D$ there exists 
$\alpha \geq \beta$ s.t.\ 
$\gamma (V,  \blds{\lambda}_{\alpha}, \mbf{x}_{\alpha}) \notin A$. (We have to resort to nets because $\msf{F}$ might not be first countable. See \eqref{E:D.metric.F.normal}.))

Recall that $D_{0} \subset D$ is ``cofinal'' in $D$ if for every $\alpha \in D$ there exists 
$\beta \in D_{0}$ s.t.\ $\beta \geq \alpha$. Let 
    \begin{equation*}
      D_{1} := \bigl\{ \alpha \in D : \blds{\lambda}_{\alpha} = (0, \ldots, 0, 1) \bigr\} . 
    \end{equation*}
Suppose $D_{1}$ is not cofinal in $D$. (For example, $D_{1}$ might be empty.)
Then there exists $\alpha_{1} \in D$ s.t.\ for no $\beta \in D_{1}$ do we have 
$\beta \geq \alpha_{1}$. We will cover that case presently. If, on the other hand, $D_{1}$ is cofinal, then 
$\bigl\{ (\blds{\lambda}_{\alpha}, \mbf{x}_{\alpha}) : \alpha \in D_{1} \bigr\}$ converges 
to $(\blds{\lambda}_{0}, \mbf{x}_{0})$. In particular, $\blds{\lambda}_{0} = (0, \ldots, 0, 1)$ and 
$x_{\alpha m} \to x_{m}$ along $D_{1}$. By \eqref{E:Gamma.combo.defn}, 
    \begin{equation*}
      \gamma \bigl( V,  (0, \ldots, 0,1), \mbf{x}_{\alpha} \bigr) = x_{\alpha m}
        \to x_{m} = \gamma \bigl( V,  \blds{\lambda}_{0}, \mbf{x}_{0} \bigr) \in A 
          \text{ along } D_{1}.
    \end{equation*}
Hence, 
    \begin{equation}  \label{E:convergence.along.D1}
      \text{there exists } \alpha_{1} \in D_{1} \text{ s.t.\ if } \beta \in D_{1} 
        \text{ with } \beta \geq \alpha_{1} \text{ then } 
          \gamma \bigl( V,  \blds{\lambda}_{\beta}, \mbf{x}_{\beta} \bigr) \in A .
    \end{equation}

Similarly, let $D_{<1} := \bigl\{ \alpha \in D : \lambda_{\alpha m} < 1 \bigr\}$. 
Suppose $D_{<1}$ is not cofinal in $D$. Then there exists $\alpha_{<1} \in D$ s.t.\ for 
no $\beta \in D_{<1}$ do we have $\beta \geq \alpha_{<1}$. I.e., eventually, $D$ and $D_{1}$ are the same. Hence, $D_{1}$ is cofinal and, from what we have just seen, 
$\gamma \bigl( V,  \blds{\lambda}_{\beta}, \mbf{x}_{\beta} \bigr)$ is eventually in $A$, i.e., it is in $A$ from some point on. Contradiction.

Therefore $D_{<1}$ must be cofinal. Hence, 
$\bigl\{ (\blds{\lambda}_{\alpha}, \mbf{x}_{\alpha}) : \alpha \in D_{<1} \bigr\}$ converges 
to $(\blds{\lambda}_{0}, \mbf{x}_{0})$. Let 
    \begin{equation*}
	y_{\alpha (m-1)} :=  \gamma \bigl[ V,  (1-\lambda_{\alpha m})^{-1} 
	  (\lambda_{\alpha 0}, \ldots, \lambda_{\alpha (m-1)}), 
	     (x_{\alpha 0}, \ldots, x_{\alpha (m-1)}) \bigr],  \qquad \alpha \in D_{<1} .
    \end{equation*}
Then by \eqref{E:Gamma.combo.defn}, 
    \begin{equation} \label{E:gamma(V,lambda,x).along.D<1}
      \gamma ( V,  \blds{\lambda}_{\alpha}, \mbf{x}_{\alpha} ) =
        \Gamma ( \lambda_{\alpha m} s_{y_{\alpha (m-1)}, x_{\alpha m}} ;  \;
          y_{\alpha (m-1)}, x_{\alpha m} ) , \qquad \alpha \in D_{<1} .
    \end{equation}

Suppose $\gamma \bigl( V,  \blds{\lambda}_{   }, \mbf{x}_{\alpha} \bigr)$ 
($\alpha \in D_{<1}$) is not eventually in $A$. Then the directed set 
    \begin{equation*}
      D_{A^{c}} := \bigl\{ \alpha \in D_{<1} : 
        \gamma ( V,  \blds{\lambda}_{\alpha}, \mbf{x}_{\alpha} ) \notin A \bigr\}
    \end{equation*}
is cofinal. For every $\alpha \in D_{A^{c}} \subset D_{<1}$, we have 
$(1-\lambda_{\alpha m})^{-1} (\lambda_{\alpha 0}, \ldots, \lambda_{\alpha (m-1)}) \in \Delta_{m-1}$. Since $\Delta_{m-1}$ is compact, it follows from Ash \cite[Theorem A5.2(c), p.\ 381]{rbA72}), that the net $\bigl\{ (1-\lambda_{\alpha m})^{-1} (\lambda_{\alpha 0}, \ldots, \lambda_{\alpha (m-1)}) : \alpha \in D_{A^{c}} \bigr\}$ has a subnet (Ash \cite[Definition A2.5, p.\ 372]{rbA72}) converging to some $\blds{\mu} \in \Delta_{m-1}$. Denote by 
$D_{ A^{c}, \blds{\mu} }$ the directed set indexing that subnet. 

By the induction hypothesis \eqref{E:gamma.induction.hyp}, 
$\{ y_{\alpha (m-1)} : \alpha \in D_{ A^{c}, \blds{\mu} } \}$ converges to \linebreak 
$y_{m-1} := \gamma \bigl( V,  \blds{\mu},  (x_{0}, \ldots, x_{(m-1)}) \bigr)$. 
Therefore, by \eqref{E:gamma(V,lambda,x).along.D<1} and hypothesis 
on $\Gamma$, 
$\gamma ( V,  \blds{\lambda}_{\alpha}, \mbf{x}_{\alpha} )$ converges, along 
$D_{ A^{c}, \blds{\mu} }$, to 
$\Gamma ( \lambda_{m} s_{y_{m-1}, x_{m}} ;  y_{m-1}, x_{m} )$. If $\lambda_{m} =1$, then 
$\Gamma ( \lambda_{m} s_{y_{m-1}, x_{m}} ;  y_{m-1}, x_{m} ) = x_{m} = \gamma ( V,  \blds{\lambda}_{0}, \mbf{x}_{0} ) \in A$. 
If $\lambda_{m} < 1$, then $\blds{\mu} = (1-\lambda_{m})^{-1} (\lambda_{0}, \ldots, \lambda_{m-1})$ and, by \eqref{E:Gamma.combo.defn}, 
$\Gamma ( \lambda_{m} s_{y_{m-1}, x_{m}} ;  y_{m-1}, x_{m} ) 
= \gamma ( V,  \blds{\lambda}_{0}, \mbf{x}_{0} ) \in A$.
In particular, for some 
$\alpha \in D_{ A^{c}, \blds{\mu} } \subset D_{A^{c}}$ we have 
$\gamma ( V,  \blds{\lambda}_{\alpha}, \mbf{x}_{\alpha} ) \in A$. This contradicts the definition of $D_{ A^{c}, \blds{\mu} }$ as a subset of $D_{<1}$. We conclude that 
$\gamma \bigl( V,  \blds{\lambda}_{\alpha}, \mbf{x}_{\alpha} \bigr)$ ($\alpha \in D_{<1}$) is eventually in $A$. Thus, there exists 
 $\alpha_{<1} \in D_{<1}$ s.t.\ if $\delta \in D_{<1}$ and $\delta \geq \alpha_{<1}$ 
 then $\gamma \bigl( V,  \blds{\lambda}_{\delta}, \mbf{x}_{\delta} \bigr) \in A$. 
 
 By \eqref{E:convergence.along.D1}, there exists  $\alpha_{1} \in D_{1}$ s.t.\ 
 if $\delta \in D_{1}$  and $\delta \geq \alpha_{1}$ 
 then $\gamma \bigl( V,  \blds{\lambda}_{\delta}, \mbf{x}_{\delta} \bigr) \in A$.
There exists $\alpha_{\leq 1} \in D$ s.t.\ $\alpha_{\leq 1} \geq \alpha_{1}$ and 
$\alpha_{\leq 1} \geq \alpha_{<1}$. 
 
Pick $\beta \in D \geq \alpha_{\leq 1}$. By assumption, there exists $\alpha \geq \beta$ s.t.\ 
$\gamma (V,  \blds{\lambda}_{\alpha}, \mbf{x}_{\alpha}) \notin A$. 
Now, $\lambda_{\alpha m}$ is either 1 or is less than 1. Hence, $\alpha \in D_{1}$ 
or $\alpha \in D_{<1}$. Suppose $\lambda_{\alpha m} = 1$. 
We have $\alpha \geq \beta \geq \alpha_{\leq 1} \geq \alpha_{1}$. Hence, by definition 
of $\alpha_{1}$, we have 
$\gamma \bigl( V,  \blds{\lambda}_{\alpha}, \mbf{x}_{\alpha} \bigr) \in A$. Similarly, 
if $\lambda_{\alpha m} < 1$. This contradicts our assumption and proves continuity 
of $\gamma(V, \blds{\lambda}, \mbf{x})$ ($\blds{\lambda} \in \Delta_{m}$ 
and $\mbf{x} \in V^{m+1}$). 

We also have that definition \ref{D:convex.combo.fn}, part (\ref{I:convex.combos.of.1.pt}) holds with 
$k = m$. 
To see this let $x \in V \in \msf{V}$. If $\lambda_{m} = 1$ then trivially, by \eqref{E:Gamma.combo.defn}, we have that \eqref{E:gamma.k.on.diagonal} holds. Now suppose $\lambda_{m} < 1$ and let $x_{0} = \cdots = x_{m} = x$. Note that, by the induction hypothesis and \eqref{E:y.k-1.gamma.defn}, 
we have $y_{m-1} = x$. Thus, by assumption, $s_{y_{m-1}, x_{m}} = s_{x, x} = 0$. Therefore, \eqref{E:Gamma.combo.defn} and \eqref{E:simple.Gamma.facts},
    \begin{equation*}
        \gamma \bigl[ V,  (\lambda_{0}, \ldots, \lambda_{m}), (x, \ldots, x) \bigr]  =
         \Gamma ( \lambda_{m} \cdot 0 ;  x, x )  = x.
    \end{equation*}

Definition \ref{D:convex.combo.fn}, part (\ref{I:1.is.identity}) $\gamma(V, 1, x_{0} ) = x_{0}$ holds by definition. 

We prove \eqref{E:drop.0.coefs} holds with $k = m$. First, assume $j=m$. I.e., $\lambda_{m} = 0$. 
Then, by \eqref{E:Gamma.combo.defn}, \eqref{E:simple.Gamma.facts}, and \eqref{E:y.k-1.gamma.defn} with $k = m-1$,
    \begin{multline*}
        \gamma \bigl[ V,  (\lambda_{0}, \ldots, \lambda_{m-1}, 0), (x_{0}, \ldots, x_{m}) \bigr]  = 
         \Gamma ( 0 \cdot s_{y_{m-1}, x_{m}} ;  y_{m-1}, x_{m} ) \\
           = y_{m-1} = \gamma \bigl[ V, (\lambda_{0}, \ldots, \lambda_{m-1}), (x_{0}, \ldots, x_{m-1}) \bigr].
    \end{multline*}
Next, suppose $j < m$. First case: $\lambda_{m} = 1$. By \eqref{E:Gamma.combo.defn}, 
    \begin{equation*}
        \gamma \bigl[ V,  (
            \underbrace{0, \ldots, 0}_{m-1}
          , 1), (x_{0}, \ldots, x_{m}) \bigr] = x_{m} = \gamma \bigl[ V,  (
            \underbrace{0, \ldots, 0}_{m-2}
          , 1), (x_{0}, \ldots, x_{j-1}, x_{j+1}, x_{m}) \bigr].
    \end{equation*}

Second case: $\lambda_{m} < 1$. By \eqref{E:y.k-1.gamma.defn} and the induction hypothesis,
   \begin{multline*}
     y_{m-1} = \\
       \gamma \bigl[ V,  (1-\lambda_{m})^{-1} 
         (\lambda_{0}, \ldots, \lambda_{j-1}, 0, \lambda_{j+1}, \ldots, \lambda_{m-1}),
           (x_{0}, \ldots, x_{j-1}, x_{j}, x_{j+1}, \ldots, x_{m-1}) \bigr] \\
             = \gamma \bigl[ V,  (1-\lambda_{m})^{-1} 
              (\lambda_{0}, \ldots, \lambda_{j-1}, \lambda_{j+1},  \ldots, \lambda_{m-1}),
                (x_{0}, \ldots, x_{j-1}, x_{j+1}, \ldots, x_{m-1}) \bigr].
   \end{multline*}
By \eqref{E:Gamma.combo.defn} with $k = m$:
    \begin{multline*}
        \gamma \bigl[ V,  (\lambda_{0}, \ldots, \lambda_{j-1}, 0, \lambda_{j+1}, \ldots,  \lambda_{k}), 
          (x_{0}, \ldots, x_{j-1}, x_{j}, x_{j+1}, \ldots,  x_{m}) \bigr] \\
            = \Gamma ( \lambda_{m} s_{y_{m-1}, x_{m}} ;  y_{m-1}, x_{m} ) \\
              = \gamma \bigl[ V,  (\lambda_{0}, \ldots, \lambda_{j-1}, \lambda_{j+1}, 
                \ldots,  \lambda_{m}), (x_{0}, \ldots, x_{j-1}, x_{j+1}, \ldots,  x_{m}) \bigr].
    \end{multline*}
$\gamma$ thus enjoys all the properties of a convex combination function 
on $\msf{V}$.
  \end{proof} 

\chapter{Singularity in Plane Fitting}  \label{Chptr:sings.in.plane.fit}
In the remainder of this book we apply the results of the previous chapters to specific classes of data maps and examine the issue of singularity therein. Let $r$ be as in theorem \ref{T:Phi.star.Hr.contains.Theta.star.Hr}. The $r = 0$ case (e.g., hypothesis testing) is very important in practice. That case is treated in example \ref{Ex:disconnected.F}. Here we examine a class of data maps, \emph{viz.} plane-fitting, for which 
    \begin{equation}  \label{E:r=1}
      r = 1.
    \end{equation} 
(The basic results of this chapter are derived using simpler arguments in \cite{spE91.sings.plane.fit,spE95}. But those arguments are specific to plane-fitting. This chapter is to show how the results of chapters \ref{Chptr:topology}, \ref{Chptr:Haus.meas.of.sing.set}, and \ref{Chptr:severity}, apply to plane-fitting.) 

Of the examples we examine, plane-fitting is probably the most difficult to analyze. 

A very common data analytic operation is fitting a plane to multivariate data. Singularity is inherent in plane fitting (Belsley \cite{daB91}, \cite{spE91.sings.plane.fit,spE95,spE96,spE98,spE.3.or.4}). Let $n =$ sample size, 
$\nvar =$ number of variables, $k =$ dimension of plane to be fitted. Assume 
	\begin{equation}  \label{E:n>nvar>k>0}
		n > \nvar > k > 0.
	\end{equation}
In the plane-fitting context a data set is an $n \times \nvar$ matrix of real numbers. In this chapter we generically denote data sets by $Y$ and we denote the set of all such data sets by $\Y$. I.e.,
    \begin{equation} \label{E:mcl.Y,defn}
      \Y := \text{ space of all } n \times \nvar \text{ real matrices.}
    \end{equation}

Thus, 
    \begin{equation}  \label{E:Y.homeomto.Rn.nvar}
      \Y \homeomto \RR^{n \nvar} ,
    \end{equation} 
where $\homeomto$ means ``homeomorphic to''. (One way to do this is to metrize $\Y$ by the Euclidean norm. See \eqref{E:matrix.norm}.) A data set $Y \in \Y$ is often referred to as a ``point cloud''. Tentatively, let $\D := \Y$. (We discuss other choices of the data space 
$\D$ in section \ref{SS:D.T.plane.fit}.) Our interest is in stability w.r.t.\ perturbations in the data in $\D$, as opposed to stability w.r.t.\, say, augmentation of the data set (Dodge and Roenko \cite{yDnR92.L1.stability}). We follow the convention of indicating the dimension of matrices by superscripts. Thus, e.g., $Y^{n \times \nvar}$ specifies that $Y$ is an $n \times \nvar$ matrix. In this section, $x$ and $y$ will usually denote vectors or numbers. (In this book all matrices and vectors are real unless otherwise specified.)  

    \begin{equation}  \label{E:rho=row.space}
      \text{If } Y \text{ is a matrix, denote its row space by } \rho(Y). 
    \end{equation}
Let 
    \begin{multline}  \label{E:1n.col.vec.defn}
      1^{n} := \text{ the $n$-dimensional column vector consisting only of $1$'s. } \\
        1_{n} := (1^{n})^{T} \text{ is an $n$-dimensional row vector} . 
    \end{multline} 
(Do not confuse this with similar notation for the indicator function introduced in \eqref{E:indicator.fn.defn}. Hopefully, which is meant will be clear from context.)
Let $G(k,\nvar)$ denote the Grassmann manifold of all $k$-dimensional linear subspaces 
of $\RR^{\nvar}$. It can be given the structure of a compact Riemannian manifold (Boothby \cite[Example 2.6, pp.\ 63--64 and Theorem (4.5), p.\ 193]{wmB75}, Wong \cite{ycW67.Grassmanif}). 
    
   \begin{remark}[Affine planes] \label{R:affine.planes}
Often one wants to ``fit'' to a point of $\Y$ an $k$-dimensional plane in $\RR^{\nvar}$. The plane does not necessarily have to pass through the origin, 0. We call such a plane a ``$k$-plane''. In this chapter we focus on maps, $\Phi : \Y \to G(k, \nvar)$. Consequently, a necessary form of ``post-processing'' is to shift $k$-planes not passing through the origin to the $k$-plane passing through the origin, i.e.\ in $G(k, \nvar)$, parallel to it. Here we show that the shifting operation is continuous.

Let $A$ denote the set of all $k$-planes in $\RR^{\nvar}$ not necessarily passing through 0.    
    \begin{multline}  \label{E:alpha=xi+x}
      \text{Each plane, } \alpha \in A \text{ has the form } \alpha = \xi + x, 
        \text{ where } \xi \in G(k, \nvar) \text{ and } x \in \RR^{\nvar} . \\
          \text{ This equality holds for any } x \in \alpha.
    \end{multline} 
(This is essentially \eqref{E:when.rows.on.alpha} below. Since $0 \in \xi$, we must have 
$x \in \alpha$. We try to adhere to the convention that vectors are row vectors unless otherwise specified.) 

We topologize $A$ as follows. Define an equivalence relation $R$ 
on $G(k, \nvar) \times \RR^{\nvar}$ as follows. If $\xi, \zeta \in G(k, \nvar)$ 
and $x,z \in \RR^{\nvar}$ write $(\xi, x) R (\zeta, z)$ if $\xi = \zeta$ and $x-z \in \xi$. 
Let $[\xi, x]$ denote the equivalence class containing $(\xi,x)$ and let $Q$ be the quotient space, 
$\bigl( G(k, \nvar) \times \RR^{\nvar} \bigr) / R$ (Munkres \cite[p.\ 112]{jrM84}). Since $(\xi,x) \mapsto \xi$ is continuous, by Munkres \cite[p.\ 112]{jrM84} again, the map $[\xi, x] \mapsto \xi$ is continuous.
The map $[\xi, x] \mapsto \xi + x$ is a well-defined bijection from $Q \to A$.\footnote{Surjectivity is obvious. What about injectivity? By \eqref{E:n>nvar>k>0}, $k < \nvar$. Let $w, y \in \RR^{\nvar}$; 
let $\xi, \zeta \in G(k, \nvar)$; and suppose $\xi + w = \zeta + y$. We need to show 
$[\xi,w] = [\zeta,y]$. It suffices to show $\xi = \zeta$, because $w-x \in \xi$ is immediate 
from $\xi = \zeta$. If $z \in \zeta$ there exists $x \in \xi$ s.t.\ 
    \begin{equation*}
      x +w = z + y .
    \end{equation*}
If $w = 0$ then take $z = 0$ leaving us with $x = y$. I.e., $y \in \xi$. It follows that
$\xi = \zeta$ and we are done. Suppose $w = 0$ is not a possibility and let $v \in \xi$ be the orthogonal projection of $w$ on $\xi$. Then $\xi + (w - v) = (\xi + v) + (w - v) = \zeta + y$ so we may assume $w \perp \xi$. Thus, $|w|^{2} = w \cdot x + |w|^{2} = w \cdot z + w \cdot y$. I.e., 
$w \cdot z = |w|^{2} - w \cdot y$. I.e., $w \cdot z$ is constant in $z \in \zeta$. 
So, e.g., for every $z \in \zeta$, we have $w \cdot z = w \cdot (2z) - w \cdot z = 0$. 
Thus, $w \perp \zeta$. Hence, both $\xi$ and $\zeta$ lie in the orthogonal complement, $w^{\perp}$, of $w$. (See \eqref{E:superscript.perp.notation}.) This allows a reduction of $\nvar$ by 1. Proceeding recursively we eventually have $\nvar = k+1$. When we carry out the operation again we have both $k$-dimensional $\xi$ and $\zeta$ in a $k$-dimensional $w^{\perp}$. This means $\xi = \zeta$.} 
Put on $A$ the topology making this bijection a homeomorphism.
   \end{remark}

Let $k' = 1, \ldots, \nvar$ and let $\alpha \subset \RR^{\nvar}$ ($\mathbb{R} =$ reals) be a plane (not necessarily passing through the origin) of dimension $k'$, a ``$k'$-plane''. That means that there is a $k'$-dimensional subspace $\zeta$ (so $\zeta$ includes the origin) 
s.t.\ $y \in \alpha$ if and only if $\alpha - y = \zeta$. 
Let $(Z')^{k' \times \nvar}$ be a matrix whose row space is $\zeta$. Hence, 
$x \in \alpha$ if and only if there is $b^{1 \times k'} \in \RR^{k'}$ s.t.\ 
$x = b Z' + y$. Now let $y' \in \alpha$. Then there exists 
$(b')^{1 \times k'} \in \RR^{k'}$ s.t.\ $y' = b' Z' + y$. Therefore, 
$x = b Z' + (y' - b'Z') = (b-b')Z' + y'$. Hence:
    \begin{multline}  \label{E:when.rows.on.alpha}
        \text{Let } \alpha \subset \RR^{\nvar} 
          \text{ be a $k'$-plane and let } y^{1 \times \nvar} \in \alpha . \\
            \text{ Let $(Z')^{k' \times \nvar}$ satisfy } 
             \rho(Z') = \alpha - y \in G(k',n) . \\
              \text{ Then the rows of } Y^{n \times \nvar} \text{ lie exactly on } \alpha 
                \text{ if and only if there exists a } \\
               B^{n \times k'} 
                 \text{ s.t.\ } Y = B Z' + 1^{n} y, \text{ in which case for every } \\
                   y' \in \alpha \text{ there exists a } (B')^{n \times k'} 
                      \text{ s.t.\ } Y = B' Z' + 1^{n} y' .             
    \end{multline}

In the following we will often have occasion to work with matrices of a certain form. Here is a basic fact about them. Let $r, s = 1, 2, \ldots$ be arbitrary. Recall that $1^{r}$ is the column vector all of whose entries are 1. We have,
    \begin{multline}  \label{E:const.rowspace.of.offset.mats} 
       \text{Let } X^{r \times s} \text{ be given. Then the map }
         w^{r \times 1} \mapsto \rho( X - 1^{r} w^{T} X ) \\
           \text{ is constant in } w^{r \times 1} \text{ s.t.\ } w^{T} 1^{r} = 1 . 
    \end{multline}
($\rho$ is the row space operator; ``$^{T}$'' = matrix transposition.) Thus, 
$w^{T} 1^{r} = 1$ means the sum of entries of $w$ is 1. E.g., $w = (1, 0, \ldots, 0)^{T}$ or $w = r^{-1} (1, 1, \ldots, 1)^{T}$, but $w$ may even have some negative entries. 
Also $1^{r} w^{T} X$ is the $r \times s$ matrix each of whose rows is the $1 \times s$ row vector $w^{T} X$. To prove \eqref{E:const.rowspace.of.offset.mats}, let $w^{r \times 1}$ satisfy $w^{T} 1^{r} = 1$ and let $W^{r \times s} := X - 1^{r} w^{T} X$. 
Notice that $w^{T} W = 0$. Let $\alpha := \rho (W)$, 
so $\alpha$ is an $s$-dimensional vector space. Let $v^{r \times 1}$ satisfy 
$v^{T} 1^{r} = 1$. Then, 
$(v^{T} - w^{T}) X = v^{T} (X - 1^{r} w^{T} X) \in \alpha$. 
Let $V^{r \times s} := X - 1^{r} v^{T} X$. Let $a^{r \times 1}$ be arbitrary. Then
    \begin{multline*}
       a^{T} V = a^{T} (V-W) + a^{T} W 
         = a^{T} \bigl[ (X - 1^{r} v^{T} X) - (X - 1^{r} w^{T} X) \bigr] + a^{T} W \\
           = (a^{T} 1^{r}) (w^{T} - v^{T}) X + a^{T} W \in \alpha + \alpha = \alpha .
    \end{multline*}
This proves that $\rho(V) \subset \alpha = \rho(W)$. Similarly, $\rho(W) \subset \rho(V)$, 
so $\rho(V) = \rho(W)$. 

Suppose the rows of $Y \in \Y$ lie exactly on a $k$-plane $\alpha \subset \RR^{\nvar}$. 
and let $w^{n \times 1}$ be a column vector s.t.\ $w^{T} 1^{n} = 1$. We have:
	\begin{multline}  \label{E:wT.Y.in.alph}   
	     \text{Suppose the rows of $Y \in \Y$ lie exactly on a $k$-plane 
		  $\alpha \subset \RR^{\nvar}$ and } w^{T} 1^{n} = 1. \\ 
		    \text{ Then } w^{T} Y \in \alpha .
	\end{multline}
By \eqref{E:alpha=xi+x}, $\alpha = \xi + y$, where $\xi \in G(k, \nvar)$ and 
$y \in \alpha$. Therefore, there is a matrix $X^{n \times \nvar}$ s.t.\ $\rho(X) = \xi$ and 
$Y = X + 1^{n} y$.
Thus, $w^{T} Y = w^{T} X + w^{T} 1^{n} y = w^{T} X + y$. But $w^{T} X \in \xi$, 
so $w^{T} X + y \in \alpha$, as desired.
Conversely, suppose \eqref{E:wT.Y.in.alph} holds for every $w^{n \times 1}$ s.t.\ 
$w^{T} 1^{n} = 1$. Taking coordinate vectors $w = (1, 0, \ldots, 0)^{T}$ etc., it is immediate that the rows of $Y$ lie on $\alpha$. 

Let 
    \begin{multline}  \label{E:Perfect.Fits.in.plane.fitting}
      \Pf = \Pf^{k} \text{ be the collection of all data sets whose rows } \\
        \text{lie exactly on a unique $k$-plane } 
         \text{(not necessarily passing through the origin).}
    \end{multline} 
(A stronger notion of perfect fit is proposed in remark \ref{R:degenerate.data}.) Thus, the data sets in $\Pf$ are ``perfect fits'' (subsection \ref{SS:calibration}) w.r.t.\ the operation of fitting $k$-planes to datasets in $\Y$. If $\xi \in G(k,\nvar)$ and 
$r \in \RR \setminus \{0\}$, then $r \xi = \xi$. The following is then immediate from \eqref{E:alpha=xi+x}:
    \begin{equation} \label{E:Pf.invar.under.shifts.rescale}
      \text{Let } s \in \RR \setminus \{0\} , \; 
        x \in \RR^{\nvar}, \text{ and } Y \in \Y . \text{ Then } Y \in \Pf 
          \text{ if and only if } s Y + 1^{n} x \in \Pf .
    \end{equation}

\emph{Claim:} 
    \begin{equation} \label{E:Y.in.Pf.means.rows.not.in.lwr.dim.plane}
       \text{If } Y \in \Pf 
         \text{ then the rows of $Y$ do \emph{not} lie on a plane of dimension } < k.
    \end{equation}
For suppose $Y \in \Pf$ and suppose the rows of $Y^{n \times \nvar}$ lie exactly on a $k'$-plane $\alpha$, where $k' < k < \nvar$. Then, by \eqref{E:when.rows.on.alpha}, if $y \in \alpha$ there exists $(B')^{n \times k'}$ 
s.t.\ $Y = B' Z' + 1^{n} y$, where $(Z')^{k' \times \nvar}$ has rank $k'$. By \eqref{E:n>nvar>k>0}, the orthogonal complement of the row space, 
$\rho(Z')$, of $Z'$ has dimension $\nvar - k' \geq k+1-k' \geq 2$. Thus, there exist nonzero orthogonal 
$\nvar$-dimensional row vectors $z_{k'+1}, \ldots, z_{k-1}, z_{k}, z_{k+1}$ perpendicular to $\rho(Z')$. Append $k-k'$ column vectors all of whose entries are 0 to the right side of $B'$ to create a $n \times k$ matrix $B$. Thus, the last $k-k'$ columns of $B$ are 0. Drop $z_{k+1}$ and append $z_{k'+1}, \ldots, z_{k-1}, z_{k}$ to the bottom of $Z'$ to create a $k \times \nvar$ matrix $Z_{k}$ and append $z_{k'+1}, \ldots, z_{k-1}, z_{k+1}$ (i.e., leaving out $z_{k}$) to the bottom of $Z'$ to create a $k \times \nvar$ matrix $Z_{k+1}$. Thus, $Z_{k}$ and $Z_{k+1}$ have different row spaces. Yet $B Z_{k} + 1^{n} y = Y = B Z_{k+1} + 1^{n} y$. Thus, by \eqref{E:when.rows.on.alpha}, the rows of $Y$ lie on two distinct $k$-planes. Therefore, $Y \notin \Pf$ and the claim \eqref{E:Y.in.Pf.means.rows.not.in.lwr.dim.plane} is proved. 

Conversely, we also \emph{claim:}
    \begin{multline}  \label{E:rows.on.k-plane.not.on.lwr.dim.plane.then.in.Pf}
      \text{If } Y \in \Y \text{ and the rows of $Y$ lie on a plane of dimension } k, \\
        \text{ but do \emph{not} lie on any plane of dimension } < k. 
          \text{ Then } Y \in \Pf. 
    \end{multline}
Let $Y \in \Y$ satisfy the hypothesis 
of \eqref{E:rows.on.k-plane.not.on.lwr.dim.plane.then.in.Pf}. We show $Y \in \Pf$. 
Let $\alpha$ be a $k$-plane containing the rows of $Y$ and let $w$ be as in \eqref{E:wT.Y.in.alph}. It follows from \eqref{E:wT.Y.in.alph} and \eqref{E:when.rows.on.alpha} that there is a $k$-dimensional subspace 
$\zeta \in G(k,\nvar)$ s.t.\ the row space 
of $X := Y - 1^{n} w^{T} Y$ is a subspace of $\zeta$. The only way $Y$ could escape membership in $\Pf$ would be if $\alpha$ were not unique, i.e.\ if there were another $k$-plane $\beta \neq \alpha$ s.t.\ the rows of $Y$ all lie on $\beta$ as well. 
But, by \eqref{E:wT.Y.in.alph}, $w^{T} Y \in \alpha \cap \beta$. Hence, by \eqref{E:when.rows.on.alpha}, there exists 
$\zeta' \in G(k,\nvar)$ s.t.\ $\zeta' \neq \zeta$ but 
$\rho(X) \subset \zeta'$. (We will not be fussy about whether planes consist of row vectors, column vectors, or just real tuples.) Thus, $\rho(X) \subset \zeta \cap \zeta'$. 
But $\dim (\zeta \cap \zeta') < k$. I.e., the rows of $Y$ \emph{do} lie on a plane of dimension $< k$, viz.\ $(\zeta \cap \zeta') + w^{T} Y$. 
This contradiction proves that $Y \in \Pf$ and hence the claim \eqref{E:rows.on.k-plane.not.on.lwr.dim.plane.then.in.Pf}.

We have
  \begin{lemma} \label{L:when.rank.Y=k.implies.Y.in.Pf}
Suppose $rank \, Y^{n \times \nvar} = k$, so the rows of $Y$ lie exactly on a plane 
in $G(k,\nvar)$. Then $Y \in \Pf^{k}$ if and only if $1^{n}$ does not lie in the column space of $Y$. 
  \end{lemma}
  \begin{proof}
Let $Y \in \Y$ have rank $k$. Let $\Xi$ be the column space of $Y$.

First, suppose $1^{n} \in \Xi$. We show $Y \notin \Pf^{k}$. 
There exists $\tilde{x}^{1 \times \nvar}$ s.t.\ $1^{n} = Y \tilde{x}^{T}$. We may assume 
$\tilde{x} \in \rho(Y)$. Let $x := |\tilde{x}|^{-2} \tilde{x}$. 
We have $\rho(Y - 1^{n} x) \subset \rho(Y)$ and
    \begin{equation*}
      ( Y - 1^{n} x ) \tilde{x} = Y \tilde{x} -   |\tilde{x}|^{-2} 1^{n} \tilde{x} \, \tilde{x}^{T} 
        = 1^{n} - 1^{n} = 0  .
    \end{equation*}
Let $\zeta$ be the orthogonal complement of $x$ \emph{within} $\rho(Y)$. 
Thus, $\dim \zeta = k-1$. (So if $k=1$ then $\zeta = \{0\}$.) The rows of $Y$ thus lie 
on $\zeta + x$, a $(k-1)$-dimensional plane. Thus, by \eqref{E:Y.in.Pf.means.rows.not.in.lwr.dim.plane}, $Y \notin \Pf$. 
(Aside: There exists $w^{n \times 1}$ s.t.\ $x = w^{T} Y$. Then 
$1 = \bigl( |\tilde{x}|^{-2} \tilde{x} \bigr) \tilde{x}^{T} 
        = x \tilde{x}^{T} = w^{T} Y  \tilde{x}^{T} = w^{T} 1^{n}$. I.e., $w^{T} 1^{n} = 1$
and $Y - 1^{n} x  = Y - 1^{n} w^{T} Y$.

Now suppose $1^{n} \notin \Xi$, $rank \, Y = k$, but $Y \notin \Pf^{k}$. 
Then, by \eqref{E:rows.on.k-plane.not.on.lwr.dim.plane.then.in.Pf}, there exists a plane 
$\alpha \subset \RR^{\nvar}$ of dimension $< k$ s.t.\ the rows of $Y$ lie exactly on $\alpha$. Hence, by \eqref{E:when.rows.on.alpha}, there exist matrices $B^{n \times (k-1)}$, $Z^{(k-1) \times \nvar}$, and 
$x^{1 \times \nvar} \in \alpha$ s.t.\ $Y = B Z + 1^{n} x$. Write $x = y + z$, 
where $y \perp \rho(Z)$ and $z \in \rho(Z)$. Thus, there exists 
$v^{(k-1) \times 1}$ s.t.\ $z = v^{T} Z$. Therefore, 
    \begin{equation*}
      Y = B Z + 1^{n} z + 1^{n} y = B Z + 1^{n} v^{T} Z + 1^{n} y 
        = (B + 1^{n} v^{T} )Z + 1^{n} y .
    \end{equation*}
If $y = 0$ then $Y = (B + 1^{n} v^{T} )Z$. But this contradicts $rank \, Y = k$. 
Therefore, $y \neq 0$ and $Y y^{T} = 0 + |y|^{2} 1^{n}$. I.e., $1^{n} \in \Xi$. Contradiction. Hence, $Y \in \Pf^{k}$. 
  \end{proof}

  \begin{example} \label{Ex:1.3.in.col.space?}
Here are two toy examples. Take $k = 1$, $\nvar = 2$, and $n = 3$, in compliance with \eqref{E:n>nvar>k>0}. First, let $Y^{3 \times 2}$ be the matrix each of whose rows is $(1,0)$. Then $rank \, Y = 1$, but each row lies on the ``0-plane'' $(1,0)$. Thus, $Y \notin \Pf^{1}$. But also $1^{3}$ is in the column space of $Y$. Next, suppose $Y^{3 \times 2}$ is the matrix with rows $(1,0)$, $(1,0)$, and $(0,0)$. Again, $rank \, Y = 1$, but the rows of $Y$ lie exactly on the unique 1-plane, the $x$-axis. Thus, $Y \in \Pf^{1}$. Moreover, $1^{3}$ is not in the column space of $Y$. 

Finally, suppose $Y^{3 \times 2}$ is the matrix with rows $(1,0)$, $(1,0)$, and $(0,1)$. Now $rank \, Y = 2$, but the rows of $Y$ lie exactly on the unique 1-plane passing through $(1,0)$ and $(0,1)$. Thus, $Y \in \Pf^{1}$. This time we again have $1^{3}$ in the column space of $Y$. But the rows of $Y$ do not lie in a line in $G(1,2)$ so this is not a counterexample to lemma \ref {L:when.rank.Y=k.implies.Y.in.Pf}.
  \end{example} 

\emph{Claim:}
    \begin{multline} \label{E:when.is.Y.in.Pk}
       \text{Let } w^{n \times 1} \text{ be some vector satisfying } w^{T} 1^{n} = 1. \\
       \text{ Then } Y \in \Pf^{k} \text{ if and only if } rank \, ( Y - 1^{n} w^{T} Y ) = k .
    \end{multline}
Let $Y \in \Y$ and $w^{n \times 1}$ satisfy $w^{T} 1^{n} = 1$. First, suppose $Y \in \Pf$. Then, by \eqref{E:Pf.invar.under.shifts.rescale}, $Y' := Y - 1^{n} w^{T} Y \in \Pf$. Let $\alpha$ be the $k$-plane on which the rows of $Y'$ lie exactly. Now, $w^{T} Y' = w^{T} Y - w^{T} 1^{n} w^{T} Y = 0$. Hence, by \eqref{E:wT.Y.in.alph}, $0 \in \alpha$. I.e., $\alpha$ passes through the origin. 
I.e., $\alpha \in G(k,n)$. The rows of $Y'$ lie exactly on $\rho(Y')$ and, by \eqref{E:Y.in.Pf.means.rows.not.in.lwr.dim.plane}, $\dim \rho(Y') \geq k$. But we must have $\rho(Y') \subset \alpha \in G(k,n)$. Hence, $\dim \rho(Y') \leq k$. 
I.e., $rank \, (Y - 1^{n} w^{T} Y) = rank \, Y' = \dim \rho(Y') = k$, as desired. 

Conversely, let $Y \in \Y$ with $rank \, ( Y - 1^{n} w^{T} Y ) = k$. By \eqref{E:Pf.invar.under.shifts.rescale}, to show $Y \in \Pf$ it suffices to show 
$Y' := Y - 1^{n} w^{T} Y \in \Pf$. By lemma \ref{L:when.rank.Y=k.implies.Y.in.Pf}, 
to show $Y' \in \Pf$ it suffices to show that $1^{n} \notin \Xi :=$ column space of $Y'$. For suppose $1^{n} \in \Xi$. Then there exists $x^{1 \times \nvar}$ s.t.\ $1^{n} = Y' x^{T}$. 
Write $w = (w_{1}, \ldots, w_{n})^{T}$, let $(y'_{i})^{1 \times \nvar}$ be the $i^{th}$ row of $Y'$, and let $a_{i} := y'_{i} x^{T} \in \RR$ ($i = 1, \dots, n$). Let $i = 1, \dots, n$. 
Now, $w^{T} 1^{n} = 1$ implies $w_{n} = 1 - (w_{1} - \cdots - w_{n-1})$. Therefore, the $i^{th}$ row 
of $(Y - 1^{n} w^{T} Y ) x^{T}$ is
    \begin{align}   \label{(Y-1.n.w.trans.Y).x.trans}
      1 = \left( y'_{i} - \sum_{j=1}^{n} w_{j} y'_{j} \right) x^{T} 
        &= a_{i} - \sum_{j=1}^{n} w_{j} a_{j} 
          \notag \\
        &= a_{i} - \sum_{j=1}^{n-1} w_{j} a_{j} - w_{n} a_{n} \\
        &= a_{i} - \sum_{j=1}^{n-1} w_{j} a_{j} - \left( 1 -  \sum_{j=1}^{n-1} w_{j} \right) a_{n} 
          \notag \\
        &= (a_{i} - a_{n}) - \sum_{j=1}^{n-1} w_{j} (a_{j} - a_{n}) . \notag
    \end{align}
This equals $1$. Thus,
    \begin{equation*}
       a_{i} - a_{n} = 1 + \sum_{j=1}^{n-1} w_{j} (a_{j} - a_{n}),  \qquad i = 1, \dots, n .
    \end{equation*}
But the RHS of the preceding is constant in $i$, which means so is the left hand side (LHS). But with $i = n$ we have 
$a_{n} - a_{n} = 0$. We conclude that $a_{i} - a_{n} = 0$ for $i = 1, \ldots, n$. But by \eqref{(Y-1.n.w.trans.Y).x.trans}, this means $1 = 0$. Contradiction. 
Therefore, $1^{n} \notin \Xi$ so by lemma \ref{L:when.rank.Y=k.implies.Y.in.Pf}, we conclude $Y \in \Pf^{k}$. This completes the proof of the claim \eqref{E:when.is.Y.in.Pk}. (The proof shows that no matter the rank of $Y - 1^{n} w^{T} Y$, $1^{n}$ is not in its column space.)

Define a map $\Delta : \Pf \to G(k,\nvar)$ as follows.
    \begin{multline}  \label{E:Delta(Y).defn}
        \text{ If } Y \in \Pf, \text{ there is a unique } k\text{-dimensional subspace } 
          \Delta(Y) := \Delta_{k}(Y) \in G(k,\nvar)  
            \\ \text{ passing through the origin, that is  parallel to the unique } \\
               k\text{-plane on which the rows of } Y \text{ lie exactly. } \\
                \text{ Let } w^{n \times 1} \text{ be an arbitrary column vector s.t.\ } 
                  w^{T} 1^{n} = 1. \\
                   \text{ Then } \Delta(Y) = \rho( Y - 1^{n} w^{T} Y ).
    \end{multline}
	
Along these lines suppose the rows of $Y \in \Pf$ lie on a $k$-dimensional subspace 
$\xi \in G(k,\nvar)$, i.e.\ a plane through the origin. 
So $\Delta(Y) = \xi$. We \emph{claim} that just what one would expect from \eqref{E:when.is.Y.in.Pk} is indeed the case:
    	\begin{multline} \label{E:Y.lies.on.k-dim.subspace.Y.in.Pf}
    	  \text{If the rows of $Y \in \Pf$ lie on a plane in } G(k,\nvar) \\
	      \text{ then there exists } w^{n \times 1} \text{ s.t.\ }
	         w^{T} 1^{n} = 1 \text{ but } w^{T} Y = 0 .
    	\end{multline} 
By lemma \ref{L:when.rank.Y=k.implies.Y.in.Pf}, $1^{n} \notin \Xi$. This means that the orthogonal projection, $(\tilde{w})^{n \times 1}$, of $1^{n}$ onto $\Xi^{\perp}$ is not 0. This means $1^{n} \tilde{w}^{T} \neq 0$. Let $w := (1^{n} \tilde{w}^{T})^{-1} \tilde{w}$. Thus, $1^{n} w^{T} = 1$. But $w \in \Xi^{\perp}$. Therefore, $w^{T} Y = 0$ as desired. This proves the claim \eqref{E:Y.lies.on.k-dim.subspace.Y.in.Pf}.

To start with, at least initially, we may identify the data space $\D$ with 
$\mathbb{R}^{\nvar n}$. In this case we say that the ``sample size'' is $n$ and the rows 
of $Y \in \Y$ are ``observations''. $\Pf$ is a manifold (\cite[p.\ 500]{spE95} or lemma \ref{L:Pf.is.a.manif}; see also \eqref{E:Pf(0).is.diff.manif}) but not a compact manifold. Its closure includes ``degenerate data sets'', viz., those lying exactly on planes of dimension $< k$. (See remark \ref{R:degenerate.data}.) We have the following. For proof, see appendix \ref{Chptr:misc.proofs}.
 
    \begin{lemma}  \label{L:Pf.is.a.manif}
	$\Pf^{k}$ is a smooth imbedded submanifold of $\Y$ of dimension 
	  $nk + (k + 1)(\nvar - k)$.
    \end{lemma}
In particular, by remark \ref{R:retraction.in.manifs}, 
    \begin{multline}  \label{E:retraction.onto.Pk}
        \text{ There exists an open neighborhood } \clU \text{ of } \Pf^{k} \\
          \text{ and a smooth retraction } R : \clU \to \Pf^{k} .
    \end{multline}

We define
    \begin{equation}  \label{E:Kk.defn}
      \mcl{K}_{k} \text{ is the space of matrices } 
        B^{k \times \nvar} \text{ of full rank } k.
    \end{equation} 
Then, by lemma \ref{L:rank.lwr.semicont} in appendix \ref{Chptr:misc.proofs},   
$\mcl{K}_{k}$ is an open subset of the space, 
$\mcl{M}_{k, \nvar} := \mcl{M} \homeomto \RR^{k \nvar}$, 
of all $k \times \nvar$ matrices. Therefore, 
    \begin{equation}  \label{E:dim.Kk=k.nvar}
      \dim \mcl{K}_{k} = k \nvar.
    \end{equation} 
Moreover, $\mcl{K}_{k}$ is also dense in $\mcl{M}$. Suppose not. 
Then $\mcl{B} := \mcl{M} \setminus \mcl{K}_{k}$ would have non-empty interior. That would mean $\dim \mcl{B} = k \nvar$. But $\mcl{B}$ consists of matrices of rank less than $k$. By lemma \ref{L:rank.k.mats.form.manif} and \eqref{E:dim.of.whole.=.max.dim.of.parts}, we have $\dim \mcl{B} < k \nvar$, contradiction. 

We always shift the plane so that it passes through the origin so the feature space is 
    \begin{equation} \label{E:feature.space.in.plane.fitting}
      \F = G(k,\nvar) .
    \end{equation}
(See remark \ref{R:affine.planes}.) Recall (Boothby \cite[Definition (2.1), p.\  60 and Example (2.6) p.\ 64]{wmB75}) that $G(k,\nvar)$ is topologized 
so that $\rho : \mcl{K}_{k} \to G(k,\nvar)$ is continuous, where as usual $\rho(X)$ is the row space a matrix $X$. Therefore,
	\begin{equation}  \label{E:convergence.in.Grassmann}
	  \text{If } B, B_{1}, B_{2}, \ldots \in  \mcl{K}_{k} 
	    \text{ and } B_{i} \to B \text{ (in $\| \cdot \|$ norm) then } \rho(B_{i}) \to \rho(B) .
	\end{equation}
Here, in accordance with our convention, the superscript ``${}^{k \times \nvar}$'' indicates matrix dimension. So ``$B^{k \times \nvar}$'' indicates that $B$ must be a matrix, of dimension $k \times \nvar$. 

Let $1 \leq j_{1} < \cdots < j_{k} \leq \nvar$ and write $\mbf{j} = (j_{1}, \ldots, j_{k})$. 
Let $\mcl{K}_{k}(\mbf{j})$ be the set of matrices $B \in \mcl{K}_{k}$ s.t.\ the $k \times k$ submatrix 
consisting of columns of $B$ indexed by $\mbf{j}$ is invertible. By lemma \ref{L:rank.lwr.semicont}, the set of $k \times k$ invertible matrices is open in the space of all $k \times k$ matrices. It follows that $\mcl{K}_{k}(\mbf{j})$ is open in $\mcl{K}_{k}$. Moreover, every matrix in $\mcl{K}_{k}$ belongs to some 
$\mcl{K}_{k}(\mbf{j})$. Let $\clU := \clU(\mbf{j}) := \rho \bigl[ \mcl{K}_{k}(\mbf{j}) \bigr]$. 

Following Boothby \cite[p.\ 64]{wmB75} (with $\nvar$ in place of $n$), we define a coordinate map 
$\varphi := \varphi_{\mbf{j}} : \clU \to \mcl{M}_{k, \nvar-k}$. (Recall that $\mcl{M}_{k, \nvar-k}$ is the space of all $k \times (\nvar - k)$ matrices.) Let $\xi \in \clU$ and let $X \in \mcl{K}_{k}(\mbf{j})$ satisfy 
$\rho(X) = \xi$. By performing row operations one can transform $X$ into a matrix 
$Z \in \mcl{K}_{k}(\mbf{j})$ whose submatrix consisting of the columns of $Z$ indexed by $\mbf{j}$ is the identity matrix $I_{k}$. Thus, $\rho(Z) = \rho(X)$. It turns out that $Z$ actually only depends on $X$ through $\xi = \rho(X)$. 
Let $V^{k \times (n-k)} \in \mcl{M}_{k, \nvar-k}$ be the matrix obtained from $Z$ by dropping its columns indexed by $\mbf{j}$. Denote the operation $X \mapsto V$ by $g$. $g$ is surjective and smooth. 
Let $h$ denote the reverse operaton that reinserts the identity matrix $I_{k}$ into $V$ to produce $Z$. Then $g \circ h$ is the identity map on $\mcl{M}_{k, \nvar-k}$ and $h \circ g$ is the identity map on $\mcl{K}_{k}(\mbf{j})$. The following commutes.
	\begin{equation*} \label{E:rho.is.smooth.on.K}
      \begin{CD}  
         \mcl{K}_{k} @<{\hookleftarrow}<< \mcl{K}_{k}(\mbf{j}) @>{\rho}>> \clU 
             @>{\hookrightarrow}>> G(k,\nvar) \\ 
         && @V{g}VV                             @AA{\rho}A  \\
         && \mcl{M}_{k, \nvar-k} @>>{h}> \mcl{K}_{k}(\mbf{j})
      \end{CD} 
	\end{equation*}
By definition of the differentiable structure on $G(k,\nvar)$, the composition $\rho \circ h$ is smooth. Therefore, the composition $\rho \circ h \circ g$ is smooth. It follows that $\rho : \mcl{K}_{k}(\mbf{j}) \to G(k,\nvar)$ is smooth. But $\mcl{K}_{k}(\mbf{j})$ is open in 
$\mcl{K}_{k}$ and every point of $\mcl{K}_{k}$ lies in some $\mcl{K}_{k}(\mbf{j})$. Hence, 
    \begin{equation} \label{E:rho.is.smooth}
      \rho : \mcl{K}_{k} \to G(k,\nvar) \text{ is smooth. }
    \end{equation}	

We can go further. Let $\mcl{K}_{n}$ be the set of all $n \times \nvar$ matrices of rank 
$k \in (0, \nvar)$. By lemma \ref{L:rank.k.mats.form.manif}, $\mcl{K}_{n}$ is an imbedded submanifold of $\Y$ (of dimension $nk + k \nvar - k^{2}$). For each choice of $0 < i_{1} < \cdots < i_{k} \leq n$ the collection of matrices in $\mcl{K}_{n}$ whose rows numbered $i_{1}, \ldots, i_{k}$ are linearly independent is open and as 
$\mbf{i} := (i_{1}, \ldots, i_{k})$ varies one gets an open cover, 
$\{ \mcl{K}_{n,\mbf{i}} \}$, of $\mcl{K}_{n}$. The operation that takes 
$Y \in \mcl{K}_{n,\mbf{i}}$ to the $k \times \nvar$ matrix, $Y_{\mbf{i}}$, consisting of its rows numbered $i_{1}, \ldots, i_{k}$ is obviously smooth. 
But $Y_{\mbf{i}} \in \mcl{K}_{k}$ and $\rho(Y_{\mbf{i}}) = \rho(Y)$. Thus, applying the last paragraph we get
	\begin{equation} \label{E:rho.is.smooth.on.nxnvar.mats}
	  \rho \text{ is smooth on the set of all } n \times \nvar 
	     \text{ matrices of rank } k.
	\end{equation} 
In particular, by \eqref{E:Delta(Y).defn}, 
    \begin{equation} \label{E:Delta.smooth.on.Pf}
      \Delta \text{ is smooth on } \Pf^{k}.
    \end{equation}

Conversely, if $\xi$ is any linear subspace of $\RR^{k}$ define
    \begin{equation}  \label{E:defn.of.Pi(xi)}
      \Pi_{\xi}^{\nvar \times \nvar} = \Pi(\xi) 
        \text{ is the matrix of orthogonal projection of } \RR^{\nvar} \text{ onto } \xi .
    \end{equation} 
Let $x \in \xi$. 
Then $x \Pi_{\xi}  = x$. If $y^{1 \times \nvar}$ is arbitrary, $y \Pi_{\xi} \in \xi$. Hence, 
$y \Pi(\xi)^{2} = y \Pi(\xi)$. It is not hard to see that $\Pi(\xi)$ is self-adjoint. Hence, by Stoll and Wong 
\cite[Theorem 6.5, p.\ 134]{rrSetW68.LinearAlgebra}, $\Pi(\xi)^{T} = \Pi(\xi)$. In summary,
    \begin{equation} \label{E:proj.mat.symmetric.idempotent}
       \Pi(\xi)^{T} = \Pi(\xi) \text{ and } \Pi(\xi)^{2} = \Pi(\xi) .
    \end{equation}
Recall \eqref{E:Kk.defn}. Then it is well-known (and easy to prove) that
	\begin{equation} \label{E:proj.mat.from.K}
		\text{If } Z^{k \times \nvar} \in \mcl{K}_{k} 
		  \text{ then } Z^{T} (Z Z^{T})^{-1} Z 
		   =  \Pi \bigl[ \rho(Z) \bigr] .
	\end{equation} 
Let $\mcl{M}$ be the space of all $\nvar \times \nvar$ real matrices. We have
    \begin{lemma} \label{L:proj.mat.is.imbedding.of.Grass}
        $\Pi$ is a smooth imbedding of $G(k, \nvar)$ into $\mcl{M}$.
    \end{lemma}
For proof see appendix \ref{Chptr:misc.proofs}. Recall that $\Pf = \Pf^{k}$ is the collection of all data sets (i.e.,  $n \times \nvar$ matrices) whose rows lie exactly on a unique $k$-plane (not necessarily through the origin). (Recall \eqref{E:Delta(Y).defn}.) Recall \eqref{E:Delta(Y).defn}. It is reasonable to demand that 
   \begin{multline}  \label{E:plane.fitting.constraint}
      \text{On dense set of data sets } Y \in \Pf, \text{a plane fitting method } \\ 
          \text{ should map } Y \text{to a plane parallel to } \Delta(Y) .
   \end{multline}
\eqref{E:plane.fitting.constraint} defines the standard, $\Sigma$, (subsection \ref{SS:calibration}) for plane fitting. In fact, we want $\Phi$ to behave well off some sort of ``null'' subset of $\Pf$. (\eqref{E:plane.fitting.constraint} will not be true if ``regularization'' is used. See remark \ref{R:regularization.generalities}.)
A data map $\Phi : \Y \partlyto G(k,\nvar)$ that satisfies \eqref{E:plane.fitting.constraint}, at least approximately, is a ``plane-fitter.'' I.e., 
    \begin{multline}  \label{E:plane-fitter.defn}
      \text{A map } \Phi : \Y \partlyto G(k,\nvar) 
        \text{ is a ``plane-fitter'' if and only if } \\
          \begin{aligned}
               1.& \; \Phi \text{ is defined and continuous on a dense subset } 
                 \Y' \text{ of } \Y, \\
               2.& \; \Y' \cap \Pf \text{ is dense in } \Pf, \text{ and } \\
               3.& \; \Phi = \Delta \text{ on } \Y' \cap \Pf  .
          \end{aligned}
    \end{multline}
Part 3 of the preceding can be relaxed. See section \ref{SS:loose.plane.fitting}.
In this chapter 
    \begin{equation}  \label{E:plane.fit.G.is.triv}
      \text{The group $G$ of homeomorphisms of $\D$ onto itself is the trivial group.}
    \end{equation}
(Common plane-fitting methods are invariant under permutation but if, say, the data are multivariate time-series, one might choose to use a method that takes order into account. The theory in this chapter allows that. In the chapters on measures of location on spheres (chapters \ref{Chptr:spherical.location}, \ref{Chptr:aug.direct.mean}, and \ref{Chptr:robst.loc.on.circle}), however, we will confine attention to methods invariant under permutation.)

\begin{example}[Three important examples] \label{Ex:3.plane.fitters}
Least squares (LS, subsection \ref{SS:lin.reg.and.LS}) and least absolute deviation regression (LAD, Bloomfield and Steiger \cite{pBwlS83}, section \ref{SS:LAD} and appendix \ref{Chptr:LAD.technicalities} below), like LS except minimizing the $L^{1}$ norm, not the $L^{2}$ norm) are well-known plane-fitting methods.

Principal components (PC, section \ref{SS:PC.plane.fitting}) is another common plane-fitter. Unlike LS and LAD, PC is an ``unsupervised learning'' method (Christianini and Shawe-Taylor \cite[p.\ 3]{nCjS-T00.svm}, Hastie \emph{et al} \cite[p.\ 438]{tHrTjF01.statlearn}, Johnson and Wichern \cite[Chapter 8]{raJdwW92}). PC is often used for dimension reduction and it works as follows. Let $Y \in \Y$. Let $Y^{0}$ be the matrix obtained from $Y$ by subtracting the arithmetic mean 
$\bar{y}^{1 \times \nvar} := n^{-1} 1_{n} Y$ of each column from all the entries in that column. So $(Y^{0})^{n \times \nvar} := Y - 1^{n} \bar{y}$. Then the ``sample covariance  matrix'' of $Y$ is 
    \begin{equation}  \label{E:sample.covariance.matrix}
      cov(Y)^{\nvar \times \nvar}  :=  \frac{1}{n-1}(Y^{0})^{T} \, Y^{0} 
        = \frac{1}{n-1} Y^{T} Y - \frac{n}{n-1} \bar{y}^{T} \bar{y} .
    \end{equation} 
(Sometimes $1/n$ is used instead of $1/(n-1)$.) The PC $k$-plane, $PC(Y)$, for $Y$ is that spanned by the eigenvectors of $cov(Y)$ corresponding to the $k$ largest eigenvalue, providing the $k^{th}$
and $(k+1)^{st}$ largest eigenvalues are unequal. (In another version, $cov(Y)$ is replaced by the correlation matrix.) \footnote{In section \ref{SSS:LF.plots}, an alternative definition of PC line-fitting on the plane. See appendix \ref{Chptr:misc.proofs} for a proof of the essential equivalence of the two definitions in this simple case.} 

Figure \ref{F:LS.PC.LAD.lf.plots} shows ``graphs'' of the three methods on a nonlinear slice through $\D$ with $k=1$, $\nvar=2$, and $n = 3$.
\end{example}

For the purpose of analyzing the singularities of plane-fitters, we sometimes replace 
\eqref{E:plane-fitter.defn} by a more general version:
   \begin{subequations} \label{E:general.plane-fitter.defn}
	\begin{gather*}
           \D \text{ satisfies \eqref{E:D.metric.F.normal}. } 
             \tag{\ref{E:general.plane-fitter.defn}a} \\
           \D \cap \Y \text{ is non-empty and the identity map on } \D \cap \Y 
              \tag{\ref{E:general.plane-fitter.defn}b} \\
            \qquad \qquad  \text{ is locally bi-Lipschitz w.r.t.\ the metrics it inherits
               from $\D$ and $\Y$.}  \\
           \D' \subset \Y \text{ and } \D' \cap \Pf^{k} \text{ is dense in } \D \cap \Pf^{k} 
             \tag{\ref{E:general.plane-fitter.defn}c} .\\
           \D' \subset \D \text{ is dense in } \D .  \tag{\ref{E:general.plane-fitter.defn}d} \\
           \text{If a test pattern space } \T \subset \Pf^{k} 
             \text{ has been specified, then } 
             \D' \cap \T \text{ is dense in } \T . \tag{\ref{E:general.plane-fitter.defn}e} \\
           \Phi : \D' \to G(k,\nvar) \text{ is continuous and } \Phi = \Delta_{k} 
             \text{ on } \D' \cap \Pf^{k} . \tag{\ref{E:general.plane-fitter.defn}f} 
	\end{gather*}
   \end{subequations}
It follows from (\ref{E:general.plane-fitter.defn}c) and (\ref{E:general.plane-fitter.defn}d) that $\Y$ is dense in $\D$. When $\D = \Y$ we often use the symbol $\Y'$ rather than 
$\D'$. In section \ref{SS:loose.plane.fitting} we show that (\ref{E:general.plane-fitter.defn}f) can be relaxed. 

\section{$\D$, $\Pf$, and $\T$ in plane-fitting} \label{SS:D.T.plane.fit}
In this section we investigate Hausdorff dimension and measure of singular sets of plane-fitting methods in general. That means investigating how plane-fitting might satisfy the hypotheses of proposition \ref{P:sing.dim.when.H.d-r.D.=.0} (and therefore of theorem \ref{T:Phi.star.Hr.contains.Theta.star.Hr}) and theorem \ref{T:lwr.bnd.on.Haus.meas}. Later sections will examine specific methods or issues. 

To apply proposition \ref{P:sing.dim.when.H.d-r.D.=.0}, $\D$ must be a compact manifold and we get the tightest bound by taking $r$ as small as possible. In fact, in this chapter   
	\begin{equation}  \label{E:r=1}
		r=1 .
	\end{equation} 
We will deal with the compactness issue later. Temporarily let $\D := \Y$ defined in \eqref{E:mcl.Y,defn}. 

The following familiar fact is useful here and below.
    \begin{equation} \label{E:switch.and.trace}
      \text{If } A \text{ and } B \text{ have dimensions } \alpha \times \beta 
        \text{ and } \beta \times \alpha, \text{ resp., then } trace \, AB = trace \, BA.
    \end{equation} 
Let $A^{\ell \times m}$ be an arbitrary real matrix. Define the \emph{Euclidean} or \emph{Frobenius} norm (Blum \emph{et al} \cite[p.\ 203]{lBfCmSsS98.realcompute}, Marcus and Minc \cite[p.\ 18]{mMhM64.MatrixThyIneq}) of $A$ by 
    \begin{equation} \label{E:matrix.norm}
      \| A \| := \sqrt{ trace \, A^{T} A } =  \sqrt{ trace \, A A^{T} } .
    \end{equation} 
($\Y$ with the Euclidean norm is consistent with \eqref{E:Y.homeomto.Rn.nvar}.)

Let $\nu_{1}, \ldots, \nu_{m}$ be the eigenvalues of $A^{T} A$ and 
$\lambda_{1}, \ldots, \lambda_{\ell}$ the eigenvalues of $A A^{T}$. We have
    \begin{equation} \label{E:Frob.norm.in.terms.of.eigvals}
      \| A \| := \sqrt{ trace \, A^{T} A } = \sqrt{ \sum_{j=1}^{m} \nu_{j} } 
        = \sqrt{ trace \, A A^{T} }  = \sqrt{ \sum_{i=1}^{\ell} \lambda_{i} } .
    \end{equation}
(All those eigenvalues are non-negative. $A A^{T}$ and $A^{T} A$ are symmetric. Diagonalize.)

Construct a preliminary version of $\T$ as follows. We will modify it below. Pick an arbitrary $(k-1)$-plane $\zeta \in G(k-1,\nvar)$ through the origin. (So if $k = 1$, $\zeta$ \emph{is} the origin.) Hold $\zeta$ fixed. Let $v_{1}, v_{2} \in \RR^{\nvar}$ be fixed orthonormal row vectors perpendicular to $\zeta$ (they exist since 
$\nvar  - \dim \zeta = (\nvar-k) + (k-\dim \zeta) \geq 2$ by \eqref{E:n>nvar>k>0}).

Define 
    \begin{equation}  \label{E:1-D.proj.space}
      P^{1} = P^{1}(\RR) \text{ is the one-dimensional real projective space }
    \end{equation}
(``real projective line''; do not confuse this ``$P$'' with ``$\Pf$''). I.e., $P^{1}$ is the space of all lines in $\RR^{2}$ through the origin (Boothby \cite[p.\ 15]{wmB75}).  
\emph{Claim:} 
    \begin{equation}  \label{E:P1.diffeom.to.S1}
      P^{1} \text{ is diffeomorphic to the circle (1-sphere), } S^{1} .
    \end{equation}
This is proved as follows. 
Let $J$ is any open interval of length $\pi$. Then 
$\ell_{J} : \theta \mapsto \text{ span of } (\cos \theta, \sin \theta) \in P^{1}$ parametrizes a coordinate neighborhood $\clU_{J}$ in $P^{1}$.\footnote{$P^{1} = G(1,2)$. Following 
Boothby \cite[Example 2.6, pp.\ 63--64]{wmB75}, $P^{1}$ should have coordinate neighborhoods defined as follows: Let $I \subset \RR$ be an open interval of length 
$\pi/2$. It is not the case that both $\sin \theta = 0$ occurs for some $\theta \in I$ 
and $\cos \theta = 0$ for some (necessarily not the same) 
$\theta \in I$. Suppose for no $\theta \in I$ do we have $\cos \theta = 0$. Then 
$\ell_{I} : \theta \mapsto \text{ span of } (1, \tan \theta) \in P^{1}$ ($\theta \in I$) is a parametrization of a coordinate neighborhood, $\clU_{I}$, of $P^{1}$. If $\cos \theta = 0$ for some $\theta \in I$ then for no $\theta \in I$ do we have $\sin \theta = 0$. In that case 
$\ell^{I} : \theta \mapsto \text{ span of } (\cot \theta, 1) \in P^{1}$ ($\theta \in I$) is a parametrization of a coordinate neighborhood, $\clU^{I}$, of $P^{1}$. (Note that if neither $\cos \theta = 0$ nor $\sin \theta = 0$ then $\ell_{I}(\theta) = \ell^{I}(\theta)$.) Clearly 
$\ell : \theta \mapsto \text{ span of } (\cos \theta, \sin \theta)$ is the same as 
$\ell_{I}$ when it is defined and $\ell^{I}$ when it is defined.}
Define $2 J := \{ 2 \phi \in \RR : \phi \in J \}$. 
The map $\omega_{J} : \phi \mapsto ( \cos \phi, \sin \phi)$, ($\phi \in 2J$) parametrizes a coordinate neighborhood, $\mcl{V}_{J}$, of $S^{1}$. Given $m \in \clU_{J}$ define 
$f_{J}(m) := \omega_{J} \bigl[ 2 \ell_{J}^{-1}(m) \bigr] \in \mcl{V}_{J}$. $f_{J}$ is smooth and so its inverse. Let $J_{1}, J_{2} \subset \RR$ be two open intervals of length 
$\pi$ and let $m \in \clU_{J_{1}} \cap \clU_{J_{2}}$. Then $\ell_{J_{1}}^{-1}(m)$ and 
$\ell_{J_{2}}^{-1}(m)$ differ by an integer multiple of $\pi$. Hence, $f_{J_{1}}$ and 
$f_{J_{2}}$ agree on $\clU_{J_{1}} \cap \clU_{J_{2}}$. The claim \eqref{E:P1.diffeom.to.S1} is proved.

If $\ell \in P^{1}$, then $\ell$ determines a line $L(\ell) \subset \RR^{\nvar}$ 
in the $(v_{1}, v_{2})$-plane:
    \begin{equation*}
      L(\ell) := \bigl\{ \alpha v_{1} + \beta v_{2} \in \RR^{\nvar} :  
       (\alpha, \beta) \in \ell \bigr\} .
    \end{equation*} 
$L(\ell)$ is a line in $\RR^{\nvar}$, i.e.\, a point in the projective space $P^{\nvar-1}$. In local coordinates $L$ is given by 
$\theta \mapsto \text{ span of } \bigl[ (\cos \theta) v_{1} + (\sin \theta) v_{2} \bigr]$, 
as $\theta$ varies over an open interval $J$ of length $\pi$. 
$L : P^{1} \to G(1, \nvar) = P^{\nvar-1}$ is clearly an imbedding.

Let 
    \begin{equation} \label{E:lambda(ell).defn}
      \lambda(\ell) := L (\ell) \oplus \zeta \in G(k,\nvar), \; (\ell \in P^{1}).
    \end{equation} 
\emph{Claim:}
    \begin{equation} \label{E:lambda(ell).imbedding}
      \lambda : P^{1} \to G(k,\nvar) \text{ is an imbedding.}
    \end{equation}
WLOG assume $v_{1} = (0, \ldots, 0, 1, 0, \ldots, 0)$, where the ``1'' is in position $q-k$ and $v_{2} = (0, \ldots, 0, 1, 0, \ldots, 0)$, where the ``1'' is in position $q-k+1$. 
Let $Z^{(k-1) \times \nvar} = (0^{(k-1) \times (\nvar-k+1)}, I_{k-1})$.  By assumption, 
$\zeta \in G(k-1,\nvar)$ and $v_{1}, v_{2} \perp \zeta$. Recall \eqref{E:rho=row.space}. WLOG $\zeta = \rho(Z)$. Let $H$ be an open interval of length $\pi/2$. Either $\sin$ does not vanish on $H$ or $\cos$ does not. Suppose it is $\cos$ that does not vanish on $H$. Then in the local coordinates of $P^{1}$, $\lambda$ can be written as 
    \begin{equation*}
      \theta \mapsto \rho
        \begin{pmatrix}
             0^{1 \times (\nvar-k-1)} & \tan \theta & 1 & 0^{1 \times (k-1)} \\
             0^{(k-1) \times (\nvar-k-1)} & 0^{(k-1) \times 1} & 0^{(k-1) \times 1} 
               & I_{k-1}
        \end{pmatrix} ,
          \qquad \theta \in H .
    \end{equation*}
The last $k$ columns of this matrix form $I_{k}$. Therefore by Boothby \cite[Example 2.6, pp.\ 63--64]{wmB75}, the preceding thus expresses $\lambda$ in local coordinates of $G(k,\nvar)$. Similarly, if it is $\sin$ that does not vanish on $H$. Hence, $\lambda$ is an immersion. It is injective. Moreover, $P^{1}$ is compact. The claim \eqref{E:lambda(ell).imbedding} now follows from Boothby 
\cite[Theorem (5.7), p.\ 79]{wmB75}.

Let $\mcl{M}$ be the space of all $\nvar \times \nvar$ real matrices. It follows from \eqref{E:lambda(ell).imbedding} and lemma \ref{L:proj.mat.is.imbedding.of.Grass} that, in the notation of that lemma, 
    \begin{equation} \label{E:Pi.circ.lambda.is.imbedding}
      \Pi \circ \lambda \text{ is an imbedding of $P^{1}$ into } \mcl{M} .
    \end{equation}

Pick $\mbf{Y} \in \Y$ of full rank, $\nvar$. The only restriction we put on $\mbf{Y}$ is that there exists 
    \begin{equation}  \label{E:plane.fitting.wT.Y=0}
      w^{n \times 1} \text{ s.t.\ } w^{T} 1^{n} = 1 \text{ and } w^{T} \mbf{Y} = 0 .
    \end{equation}
In particular, we forbid $\mbf{Y}$ from having a column proportional 
to $1^{n}$. This is possible: Start with $w^{n \times 1}$ with $w^{T} 1^{n} = 1$. The dimension of the orthogonal complement, $w^{\perp}$, (see \eqref{E:superscript.perp.notation}) of $w$ in $\RR^{n}$ has dimension $n-1 \geq \nvar$, by \eqref{E:n>nvar>k>0}. Thus, we can find $\nvar$ linearly independent $n$-vectors in $w^{\perp}$ to be the columns of $\mbf{Y}$.  

For $\ell \in P^{1}$ let $\Upsilon(\ell) \in \Y$ be the data set  ($n \times \nvar$ matrix) whose $i^{th}$ row is the orthogonal projection of the $i^{th}$ row of $\mbf{Y}$ onto $\lambda(\ell)$ ($i=1, \ldots, n$). Thus, 
    \begin{equation}  \label{E:plane-fitting.Upsilon.defn}
      \Upsilon = \mbf{Y} \, (\Pi \circ \lambda) .
    \end{equation}
It follows from \eqref{E:Pi.circ.lambda.is.imbedding} that 
    \begin{equation}  \label{E:Upsilon.is.smooth.imbedding}
       \Upsilon \text{ is a smooth imbedding of $P^{1}$ into } \Y .
    \end{equation}

Recall the definition, \eqref{E:Perfect.Fits.in.plane.fitting}, of $\Pf^{k}$. In this section, in accordance with that definition, we write $\Pf = \Pf^{k}$. \emph{Claim:} 
	\begin{equation} \label{E:Upsilon.maps.P1.into.Pf}
		\text{If } \ell \in P^{1} \text{ then } \Upsilon(\ell) \in \Pf. 
	\end{equation}
To see this, let $w^{n \times k}$ be as in \eqref{E:plane.fitting.wT.Y=0}. Thus,  $w^{T} \Upsilon(\ell) = 0$ so $\Upsilon(\ell) - 1^{n} w^{T} \Upsilon(\ell) = \Upsilon(\ell)$. Hence, by \eqref{E:when.is.Y.in.Pk}, to prove \eqref{E:Upsilon.maps.P1.into.Pf}, it suffices to prove 
$rank \, \Upsilon(\ell) = k$. First observe that, since $\dim \lambda(\ell) = k$, we have 
$rank \, \Upsilon(\ell) \leq k$. But, since $\mbf{Y}$ has full rank $\nvar$, there is $\nvar \times \nvar$ matrix, $M$, consisting of linearly independent rows of $\mbf{Y}$. I.e., $M^{\nvar \times \nvar}$ is invertible. Hence, if $x \in \lambda(\ell)$, thought of as a $1 \times \nvar$ row vector, is arbitrary, there exists $y^{1 \times \nvar}$ s.t.\ $x = y M $.
Since $x \in \lambda(\ell)$, $x = x \Pi \bigl[ \lambda(\ell) \bigr] = y M \Pi \bigl[ \lambda(\ell) \bigr] = y M(\ell)$, where $M(\ell)$ is the row-wise orthogonal projection of $M$ onto $\lambda(\ell)$. But $M(\ell)$ consists of rows of $\Upsilon(\ell)$). This proves that 
    \begin{equation}  \label{E:lambda(ell).subset.Oops!(ell)}
      \lambda(\ell) \subset \rho \bigl[ \Upsilon(\ell) \bigr] .
    \end{equation}
So $rank \, \Upsilon(\ell) \geq k$. 
Thus, 
    \begin{equation}  \label{E:plane.fit.rank.Oops!=k}
      rank \, \Upsilon(\ell) = k .
    \end{equation} 
The claim \eqref{E:Upsilon.maps.P1.into.Pf} follows. 

By \eqref{E:lambda(ell).subset.Oops!(ell)}, 
\eqref{E:plane.fit.rank.Oops!=k}, \eqref{E:Upsilon.maps.P1.into.Pf}, \eqref{E:Delta(Y).defn}, and \eqref{E:plane.fitting.wT.Y=0}, 
	\begin{equation}  \label{E:lambda(ell)=Delta.Upsilon(ell)}
		\lambda(\ell) = \rho \bigl[ \Upsilon(\ell) \bigr] 
		  = \Delta \bigl[ \Upsilon(\ell) \bigr] , \quad \ell \in P^{1}.
	\end{equation}

Let 
	\begin{equation}  \label{E:plane.fittin.T.defn}
		\T := \Upsilon(P^{1}) \subset \Pf.   
	\end{equation}
By \eqref{E:Upsilon.is.smooth.imbedding},
    \begin{equation}  \label{E:in.plane.fitting.T.is.smooth.submanif}
        \text{ $\T$ is a compact smooth 1-dimensional 
          imbedded submanifold of } \Y.
    \end{equation}
In fact, by \eqref{E:P1.diffeom.to.S1}, $\T$ is diffeomorphic to a circle ($S^{1}$). 

Let $\Theta$ be the restriction $\Delta \restriction_{\T}$. \emph{Claim:} 
    \begin{equation}  \label{E:Theta.ast.nontriv.in.dim.1}
      \Theta_{\ast} \text{ is nontrivial in dimension 1.}
    \end{equation}

$\Phi$ also satisfies \eqref{E:nontriv.r-dim.homol} with $r = 1$ (using 
$\mathbb{Z}/2$ coefficients). This is a consequence of the following. (See the appendix \ref{Chptr:misc.proofs} for proof.)

  \begin{lemma}  \label{L:nontriv.homol.if.cov.by.bndl.map}
Let $\psi_{j}$ be an $\RR^{n}$-bundle over a base space, $B_{j}$ ($j = 1,2$) 
(Milnor and Stasheff \cite[p.\ 13]{jwMjdS74}). Let $f : B_{1} \to B_{2}$ be continuous and suppose $f$ can be covered by a bundle map from $\psi_{1}$ to $\psi_{2}$ (Milnor and Stasheff \cite[p.\ 26]{jwMjdS74}). If $\psi_{1}$ has a nontrivial Stiefel-Whitney characteristic cohomology class (Milnor and Stasheff \cite[\S 4]{jwMjdS74}) in dimension $s > 0$, 
then $f_{\ast} : H_{s}(B_{1} ; \mathbb{Z}/2) \to H_{s}(B_{2} ; \mathbb{Z}/2)$, the induced homomorphism of homology in dimension $s$, is nontrivial.
  \end{lemma}

We apply this as follows. Let $\gamma^{1}_{1}$ be the canonical line bundle over $P^{1}$ 
(Milnor and Stasheff \cite[Example 4, pp.\ 15 -- 16]{jwMjdS74}). Let $\epsilon$ be the trivial $(k-1)$-bundle over $P^{1}$ all of whose fibers equal the $(k-1)$-plane $\zeta$ (Milnor and Stasheff \cite[Example 1, p.\ 14]{jwMjdS74}). Let $\psi_{1}$ be the Whitney sum, $\gamma^{1}_{1} \oplus \epsilon$ (Milnor and Stasheff \cite[p.\ 27]{jwMjdS74}). 
Thus, $\psi_{1}$ is a $k$-bundle over $P^{1}$. 

Let $\gamma^{k}(\RR^{\nvar})$ be the canonical $k$-plane bundle over $G(k,\nvar)$ (Milnor and Stasheff \cite[pp.\ 59-60]{jwMjdS74}). Now, by \eqref{E:lambda(ell).defn}, 
$\lambda(\cdot) : P^{1} \to G(k,\nvar)$ takes $\ell$ to $L (\ell) \oplus \zeta$.
Then $\lambda$ is covered by the bundle map, $F$, from $\psi_{1}$ 
to $\gamma^{k}( \RR^{\nvar} )$ defined as follows. Recall that $v_{1}, v_{2} \in \RR^{\nvar}$ are perpendicular to $\zeta$. Let $\ell \in P^{1}$, 
$(\alpha, \beta) \in \ell \subset \RR^{2}$, and $z \in \zeta \subset \RR^{\nvar}$. 
Then $F$ takes $\bigl( \ell, (\alpha, \beta, z) \bigr)$ 
to $\bigl( \lambda(\ell), \alpha v_{1} + \beta v_{2} + z \bigr)$. 
Thus, $F : \psi_{1} \to \gamma^{k}(\RR^{\nvar})$ is a bundle map that covers $\lambda$. 
But, by Milnor and Stasheff \cite[Axiom 4, p.\ 38]{jwMjdS74} 
$\gamma^{1}_{1}$ has a nontrivial Stiefel-Whitney class in dimension 1. 
Therefore, by (Milnor and Stasheff \cite[Proposition 3, p.\ 39]{jwMjdS74}), 
so does $\psi_{1}$. 
Hence, by lemma \ref{L:nontriv.homol.if.cov.by.bndl.map}, 
    \begin{equation}  \label{E:lambda.ast.nontriv}
      \lambda_{\ast} \text{ is nontrivial in dimension 1.}
    \end{equation} 
But, by \eqref{E:lambda(ell)=Delta.Upsilon(ell)}
$\lambda_{\ast} = \Delta_{\ast} \circ \Upsilon_{\ast}$. Therefore, 
$\Delta_{\ast}$ is nontrivial in dimension 1, as claimed.
 
As remarked already, $\D = \Y$ is not a compact manifold, but proposition \ref{P:sing.dim.when.H.d-r.D.=.0} and theorem \ref{T:lwr.bnd.on.Haus.meas} both require 
$\D$ to be one.  We propose two compact manifold versions of $\D$, with corresponding versions of $\T$ and $\Pf$.

\emph{First, take $\D$ to be} $\D_{\infty}$, the one point compactification of 
$\Y = \RR^{n \nvar}$. 
    \begin{equation}  \label{E:plane-fitting.D.infty}
      \D_{\infty} = \text{ the one point compactification of } \Y = \RR^{n \nvar} .
    \end{equation}
Working in $\D_{\infty}$ allows us to compute a global lower bound on the dimension of the singular set.

We parametrize $\D_{\infty} \setminus \{ \infty \}$ by inverse stereographic projection 
(Apostol \cite[p.\ 11]{tmA57.Apostol}, Edelsbrunner 
and Harer \cite[pp.\ 64--65]{hEjlH10.CompTopol}). To make this concrete, 
let $w := (0^{1 \times n \nvar}, 1) \in \RR^{n \nvar +1}$ and let $\D_{\infty}$ 
be the $n \nvar$-dimensional sphere 
$S^{n \nvar}(w) \subset \RR^{n \nvar+1}$ with center at $w$ and radius 1. We can think of $S^{n \nvar}(w)$ as ``resting'' on $\RR^{n \nvar}$ at $0 \in \RR^{n}$.
Let $z := (0^{1 \times n \nvar}, 2)$ be the North Pole of $\D_{\infty}$. $z$ will play the role of $\infty$ in the stereographic projection. Identify $\Y$ with 
$\D_{\infty} \setminus \{z\}$ using the inverse, 
$PS : \RR^{n \nvar} \to \D_{\infty} \subset \RR^{n \nvar +1}$, of a version of stereographic projection from $z$. Recall \eqref{E:matrix.norm}. Let $Y \in \Y = \RR^{n \nvar}$ and let 
$\sigma(Y) = 4/\bigl( \|Y\|^{2} + 4 \bigr)$. Then $PS(Y)$ is given by 
    \begin{equation}  \label{E:PS.formla}
      PS(Y) := \sigma(Y) \, (Y,0) + \bigl( 1-\sigma(Y) \bigr) z 
        = \Bigl( \sigma(Y) Y, 2 \bigl(1-\sigma(Y) \bigr) \Bigr) 
          \in \D_{\infty} \setminus \{z\} .
    \end{equation} 
$PS$ is clearly injective. Iits inverse, being also rational, is continuous.   

\emph{Claim:} $PS : \Y \to \RR^{n \nvar +1}$ is an imbedding (Boothby \cite[Definition (4.11), p.\ 73]{wmB75}). Let $Y \in \Y$. Let $z^{1 \times n \nvar}$ be the row vector obtained by concatenating the rows of $Y$ in order. It is not hard to see that the first $n \nvar$ rows of the Jacobian matrix $D \, PS$ (Boothby \cite[p.\ 26]{wmB75}) form the matrix
    \begin{equation*}
      \begin{pmatrix}
         \sigma(z) \left[ - \frac{2}{|z|^{2} + 4} z^{T}  z + I_{n \nvar} \right] 
      \end{pmatrix}^{(n \nvar) \times (n \nvar)} ,
    \end{equation*}
where $I_{n \nvar}$ is the $(n \nvar) \times (n \nvar)$ identity matrix. The last row is $\sigma(z)^{2} z$. Thus, $(D \, PS) z^{T}$ is an $(n \nvar+1)$-column vector whose last entry is  $\sigma(z)^{2} z$. So if $z \neq 0$ then the last coordinate of $(D \, PS) z^{T}$ is 
$\sigma(z)^{2} |z|^{2} \neq 0$. If $w \in z^{\perp}$ is non-zero then the first $n \nvar$ coordinates of $(D \, PS) w^{T}$ form the vector $\sigma(z) w \neq 0$. This proves that 
$D \, PS$ has full rank so $PS$ is an immersion. We have already observerd that $PS$ is one-to-one with a continuous inverse. I.e., $PS$ is an injective immersion with a continuous inverse. Therefore, $PS$ is an imbedding (Boothby \cite[Definition (4.11), p.\ 73]{wmB75}), as claimed. Taking note of Boothby \cite[Theorem (5.5), p.\ 78]{wmB75}, we have 
    \begin{multline}  \label{E:PS.is.diffeom}
      PS : \RR^{n \nvar} \to \D_{\infty} \setminus \{z\} \\
        \text{ is a homeomorphism, in fact, a diffeomorphism onto its image.}
    \end{multline}

Use $PS$ to identify $\Y$ with $\D_{\infty} \setminus \{z\}$ and identify 
a data map $\Phi$ on $\Y$ with $\Phi \circ PS^{-1}$ on $\D_{\infty}$. So we think 
of $\D_{\infty}$ as $\Y \cup \{\infty\}$. But under this identification, $\Y$ has the metric $(PS)^{\ast} \delta$ as in lemma \ref{L:imbedding.is.loc.Lip}, where $\delta$ is the great circle geodesic metric on $\D_{\infty}$. I.e., if $Y, Y' \in \Y$ then 
$\bigl( (PS)^{\ast} \delta \bigr)(Y, Y') = \delta \bigl[ PS(Y), PS(Y') \bigr]$. With this identification $\D_{\infty}$ satisfies (\ref{E:general.plane-fitter.defn}a) and (\ref{E:general.plane-fitter.defn}b). By \eqref{E:n>nvar>k>0} and construction, 
    \begin{equation}  \label{E:plane.fit.D.infty.is.a.sphere}
      \D = \D_{\infty} \text{ is a sphere of dimension } 
        d_{\infty} := n \nvar \geq 6 .
    \end{equation} 

Recall \eqref{E:plane-fitting.Upsilon.defn} and \eqref{E:Delta(Y).defn}. Define
    \begin{equation}  \label{E:Upsilon.infty.defn}
      \Upsilon_{\infty} := PS \circ \Upsilon .
    \end{equation}

Recall \eqref{E:1-D.proj.space}. Let 
    \begin{equation}  \label{E:T.infty:=PS[Upsilon(P1)]}
      \T_{\infty} := \Upsilon_{\infty} (P^{1}) := PS \circ \Upsilon(P^{1}) , 
        \text{ which we identify with } \Upsilon(P^{1}). 
    \end{equation}
Then, by \eqref{E:Upsilon.is.smooth.imbedding} and \eqref{E:PS.is.diffeom}, 
$\T_{\infty}$ is an imbedded one-dimensional submanifold of $\D_{\infty}$. Let 
    \begin{equation}  \label{E:Pf.infty.Delta.defns}
      \Pf_{\infty} := PS(\Pf) \text{ and }
      \Delta_{\infty} := \Delta \circ PS^{-1} \text{ on } \Pf_{\infty} .
    \end{equation} 
It follows from \eqref{E:PS.is.diffeom} and \eqref{E:Delta.smooth.on.Pf}, that
    \begin{equation}  \label{E:Delta.infty.is.cont}
      \Delta_{\infty} \text{ is continuous on } \Pf_{\infty} .
    \end{equation}

Identify $\Pf_{\infty}$ with $\Pf$. It follows from \eqref{E:lambda(ell)=Delta.Upsilon(ell)} that
    \begin{equation}  \label{E:lambda(ell)=Delta.infty.Upsilon.infty(ell)}
      \lambda(\ell) 
        = \Delta_{\infty} \bigl[ \Upsilon_{\infty}(\ell) \bigr] , \quad \ell \in P^{1}.
    \end{equation}

By lemma \ref{L:Pf.is.a.manif}, corollary \ref{C:cont.diff.=.loc.Lip}, and the fact that 
$PS$ is a diffeomorphism, we have, 
    \begin{equation}  \label{E:Pf.infty.is.submanif}
      \Pf_{\infty} \text{ is an imbedded submanifold of } \D_{\infty}
        \text{ of dimension } nk + (k + 1)(\nvar - k) . 
    \end{equation} 
(Alternatively, since $PS$ and its inverse are both rational functions, by example \ref{Ex:ratnl.fns.loc.Lip},  they are locally Lipschitz. Therefore lemma \ref{L:loc.Lip.image.of.null.set.is.null} tells us that the dimensions of subsets of $\Y = \RR^{n \nvar}$ and, in the other direction, subsets 
of $\D_{\infty} \setminus \{ \infty \}$, are preserved by $PS$ and $PS^{-1}$, resp.) 

Neither $\T_{\infty}$ nor $\Pf_{\infty}$ contain $\infty = z$.

Suppose $(\Phi, \Y')$ satisfies \eqref{E:plane-fitter.defn}. Then, by \eqref{E:PS.is.diffeom}, 
    \begin{equation}  \label{E:plane.fit.D'.infty}
      \D' = \D'_{\infty} := PS(\Y') \text{ is dense in } \D_{\infty}  
    \end{equation}
and $\D_{\infty}' \cap \Pf_{\infty}$ is dense in $\Pf_{\infty}$. Identify 
$\D'$ with $\Y'$, so we tacitly replace $\Phi$ by 
    \begin{equation}  \label{E:plane.fitting.Phi.infty}
      \Phi_{\infty} := \Phi \circ PS^{-1}
    \end{equation} 
and leave $\Phi(z)$, where $z$ is the North Pole, undefined. 
Assume $\D' \cap \T$ is dense in $\T$. Then $\D_{\infty}' \cap \T_{\infty}$ 
is dense in $\T_{\infty}$. Thus, with $\D = \D_{\infty}$ identified with 
$\Y \cup \{\infty\}$, and $\T = \T_{\infty}$, all the conditions \eqref{E:general.plane-fitter.defn} are satisfied. 

Let  
    \begin{equation}  \label{E:plane.fit.S.infty.defn}
      \Ss_{\infty} := (\D_{\infty} \setminus \D_{\infty}') \cup \{ \infty \} .
    \end{equation}  
where ``$\infty$'' denotes the point at infinity (concretely, the point 
$z = (0^{1 \times n \nvar}, 2)$). 
Since $\dim \{z\} = 0$, by \eqref{E:dim.of.whole.=.max.dim.of.parts}, 
if $\Ss \neq \varnothing$, we have $\dim \Ss = \dim \Ss_{\infty}$ and if $a > 0$ 
then $\Hm_{\D_{\infty}}^{a}(\Ss) = \Hm_{\D_{\infty}}^{a}(\Ss_{\infty})$ (calculated using the metric $\delta$ on $\D_{\infty} \subset \RR^{n \nvar+1}$).  

We have observed that subsets of $\D_{\infty} \setminus \{ \infty \}$, are preserved by $PS$ and $PS^{-1}$, resp. However, Hausdorff measure on $\D_{\infty} \setminus \{ \infty \}$ and on $\Y = \RR^{n \nvar}$, as in \eqref{E:Hm.a.S.geq.R.d-p-1}, will \emph{not} be preserved. (See remark \ref{R:different.measures.on.D}.).

\emph{As an alternative to $\D_{\infty}$ we now consider} ``$\D_{\mu}$'': 
Let $S^{n \nvar-1} := \bigl\{ Y \in \Y :  \| Y \| = 1 \bigr\}$. (See \eqref{E:matrix.norm}.) Let 
    \begin{equation}  \label{E:mu.as.in.D.mu.defn}
      \mu : S^{n \nvar-1} \to (0, \infty)
    \end{equation} 
be continuous (e.g., constant). We will assume
    \begin{equation}  \label{E:mu.is.smooth}
      \mu \text{ is smooth. }
    \end{equation}
In place of $\D = \Y$ define 
	\begin{equation}  \label{E:plane.fitting.D.mu.defn}
	      \D := \D_{\mu} := \Bigl\{ Y \in \Y \setminus \{ 0 \} : 
	           \| Y \| = \mu \bigl( \| Y \|^{-1} Y \bigr) \Bigr\}.
	\end{equation} 

Let $R_{\mu} : \Y \setminus \{0\}$ be the map 
    \begin{equation}  \label{E:R.mu.defn}
      R_{\mu} : Y \mapsto \mu \bigl( \| Y \|^{-1} Y \bigr) \| Y \|^{-1} Y 
        \in \D_{\mu}, \qquad Y \in \Y \setminus \{0\} .
    \end{equation}
So 
    \begin{equation}  \label{E:R.mu.is.retraction}
      R_{\mu} \text{ is a retraction of $\Y \setminus \{0\}$ onto } \D_{\mu} .
    \end{equation} 
Note that 
    \begin{equation}  \label{E:R.mu.is.bijection}
      \text{The restriction }R_{\mu} \restriction_{S^{n \nvar - 1}} 
        \text{ is a bijection of $S^{n \nvar - 1}$ onto } \D_{\mu}
    \end{equation}
and, by \eqref{E:Pf.invar.under.shifts.rescale} and \eqref{E:Delta(Y).defn}, 
    \begin{equation}  \label{E:R.mu(Pf)=Pf} 
      R_{\mu}(\Pf) \subset \Pf \text{ and if }Y \in \Pf \text{ then } 
        \Delta \bigl[ R_{\mu}(Y) \bigr] = \Delta(Y) .
    \end{equation}

By \eqref{E:n>nvar>k>0}, we have 
    \begin{equation}  \label{E:dim.D.mu.in.plane.fit}
      \dim \D_{\mu} = d_{\mu} := n \nvar - 1 \geq 5 ,
    \end{equation}

\emph{Claim:} 
    \begin{multline}  \label{E:D.mu.diffeo.to.sphere}
      \text{If } \mu : S^{n \nvar-1} \to (0, \infty) \text{ is smooth then } \\
      \D_{\mu} \text{ is an imbedded $C^{\infty}$ submanifold of } \Y \\
        \text{ and } R_{\mu} \restriction_{S^{d_{\mu}}} : S^{d_{\mu}} \to \D_{\mu}
          \text{ is a diffeomorphism.}
    \end{multline}
(See \eqref{E:restriction.notation}.) To see this, let $f : \Y \setminus \{ 0 \} \to \RR$ be the function
    \begin{equation}  \label{E:f:mclY.to.RR}
      f(Y) = \| Y \| - \mu \bigl( \| Y \|^{-1} Y \bigr) \in \RR, \qquad (Y \in \Y) .
    \end{equation}
Notice that $f$ has full rank 1: The derivative of $f$ at $Y$ in the direction $tY$ is 
$d f( t Y )/dt = \| Y \| > 0$. But $\D_{\mu} = f^{-1}(0)$. Hence, by Boothby \
\cite[Theorem (5.8), p.\ 79; Definition (5.3), p.\ 77; and Lemma (5.2), p.\ 76]{wmB75}, 
    \begin{equation}  \label{E:D.mu.is.imbedded}
      \D_{\mu} \text{ is an imbedded submanifold of $\Y$ of dimension } n \nvar - 1 .
    \end{equation} 
 
Regard elements of $\Y$ as $n \nvar \times 1$ column vectors $z$. Write 
    \begin{equation}  \label{E:x(z).in.sphere}
      x^{n \nvar \times 1} := x(z) :=  |z|^{-1} z \in S^{n \nvar-1} , 
        \quad z^{n \nvar \times 1} \in \RR^{nq} \setminus \{0\} .
    \end{equation} 
Define $\nabla \mu(x)^{1 \times n \nvar}$ to be the row vector with the following property. 
Let $v^{1 \times n \nvar} \in T_{x} S^{n \nvar-1}$, the tangent space to $S^{n \nvar-1}$ at $x$. So $v$ is expressed as a vector in $\RR^{n \nvar}$. At $x$, the derivative of $\mu$ in the direction $v$ is linear in $v$. Hence, there exists 
$\nabla \mu(x)^{1 \times n \nvar} \in \RR^{n \nvar}$ s.t.\ $\nabla \mu(x) v^{T}$ is the derivative of $\mu$ in the direction $v$ at $x$. WLOG we may take 
$\nabla \mu(x) \in T_{x} S^{n \nvar-1}$. 

By \eqref{E:R.mu.defn} and \eqref{E:mu.is.smooth}, 
$R_{\mu} : \Y \setminus \{0\} \to \RR^{n \nvar}$ is smooth. Let $DR_{\mu}(z)^{n \nvar \times n \nvar}$ be the Jacobian matrix of $R_{\mu}$, \emph{viz.} 
$(\partial R_{\mu,i}/\partial z_{j})$. The proof of the following is in appendix \ref{Chptr:misc.proofs}. 
    \begin{equation}  \label{E:DR.mu(z).1st.version}
      DR_{\mu}(z)^{n \nvar \times n \nvar} = 
        |z|^{-1} x \, \nabla \mu(x) 
          + \frac{\mu(x)}{|z|^{3}} \bigl( - z z^{T} + |z|^{2} I_{n \nvar} \bigr) , 
        \qquad z^{1 \times n \nvar} \in \RR^{nq} \setminus \{0\} .
    \end{equation}
 
Identify $\Y$ with $\RR^{n \nvar}$. Recall that by \eqref{E:x(z).in.sphere}, 
$x \in S^{n \nvar -1}$. If $v^{1 \times n \nvar} \neq 0$ is tangent to $S^{n \nvar -1}$ at 
$z \in S^{n \nvar -1}$, then $x = x(z) = z$, $v x = v z = 0$, and 
$v \, DR_{\mu}(z) = 0 + \mu(x) |z|^{-1} v \neq 0$. 
Hence, $DR_{\mu}$ has full rank on $T S^{n \nvar -1}$. 
It then follows from \eqref{E:D.mu.is.imbedded}, \eqref{E:R.mu.is.bijection}, and Boothby \cite[Exercise 2, p.\ 74]{wmB75} that the restriction of $R_{\mu}$ to $S^{n \nvar - 1}$ is in fact a diffeomorphism. That proves the claim \eqref{E:D.mu.diffeo.to.sphere}. 

Recall \eqref{E:matrix.norm}. Denote by $Euc$ the Euclidean metric on $\Y$: 
$Euc(Y,Y') := \|Y-Y'\|$ ($Y,Y' \in \Y$). Let $inc : \D_{\mu} \hookrightarrow \Y$ denote inclusion. In the notation of lemma \ref{L:imbedding.is.loc.Lip}, the restriction 
of $Euc$ to $\D_{\mu}$ is 
$inc^{\ast}(Euc)$: $inc^{\ast}(Euc)(x, x') = Euc(x,x')$ ($x,x' \in \D_{\mu}$). 
Put on $\D_{\mu}$ a metric $\delta$, say, s.t.\ 
    \begin{multline}  \label{E:metrics.on.D.mu}
      \text{The identity map on $\D_{\mu}$ is bi-Lipschitz w.r.t.\ } \delta 
        \text{ and } inc^{\ast}(Euc) .
    \end{multline}   
$\D = \D_{\mu}$ satisfies (\ref{E:general.plane-fitter.defn}a) and, by \eqref{E:metrics.on.D.mu}, it also satisfies (\ref{E:general.plane-fitter.defn}b).

Put on $\D_{\mu}$ the Riemannian metric, $\eta$, induced by the Euclidean metric 
on $\Y$ by the inclusion $inc : \D_{\mu} \hookrightarrow \Y$.  ($\eta$ is not the same as $inc^{\ast}(Euc)$!) Let $Eta$ be the topological metric on $\D_{\mu}$ induced by the tensor $\eta$. By lemma \ref{L:imbedding.is.loc.Lip} and \eqref{E:local.Lip.is.Lip.on.compacts}, 
the identity map on $\D_{\mu}$ is bi-Lipschitz w.r.t.\ $Eta$ and $inc^{\ast}(Euc)$. Thus, \eqref{E:D.mu.diffeo.to.sphere} implies \eqref{E:metrics.on.D.mu} with $\delta = Eta$. 

Define $f$ as in \eqref{E:f:mclY.to.RR}. Let
    \begin{equation}  \label{E:Pf.mu.defn}
      \Pf_{\mu} := \Pf \cap \D_{\mu} = \bigl\{ Y \in \Pf : f(Y) = 0 \bigr\} .
    \end{equation}
By \eqref{E:Pf.invar.under.shifts.rescale} and \eqref{E:R.mu.is.retraction},
    \begin{equation}  \label{E:Pf.mu.=.R.mu(Pf)}
      \Pf_{\mu} = R_{\mu}(\Pf) . 
    \end{equation}

As observed above, the derivative of $f$ at $Y$ in the direction $tY$ is 
$d f( t Y )/dt = \| Y \| > 0$. By \eqref{E:Pf.invar.under.shifts.rescale}, rescaling (i.e., $Y \mapsto tY$) maps $\Pf$ into itself. Therefore, $f$ has full rank on $\Pf$. Hence, by lemma \ref{L:Pf.is.a.manif} and Boothby \cite[Theorem (5.8), p.\ 79]{wmB75} again, 
    \begin{multline}  \label{E:Pf.mu.is.manifold}
      \Pf_{\mu} \text{ is a imbedded submanifold of } \Pf \text{ of dimension } \\
        \dim \Pf- 1 = nk + (k + 1)(\nvar - k) - 1 .
    \end{multline}. 

Recall  
\eqref{E:Upsilon.maps.P1.into.Pf} 
and \eqref{E:R.mu(Pf)=Pf}. Define a rescaled version of $\Upsilon(\ell)$ by 
    \begin{equation}  \label{E:Upsilon.Delta.mu.defn}
      \Upsilon_{\mu} := R_{\mu} \circ \Upsilon : P^{1} \to \Pf_{\mu} , \text{ and }
        \Delta_{\mu} := \Delta .
    \end{equation} 
We have $\| \Upsilon_{\mu}(\ell) \| = 
\mu \bigl( \bigl\| \Upsilon(\ell) \bigr\|^{-1} \Upsilon(\ell) \bigr)$. (By \eqref{E:plane.fit.rank.Oops!=k}, $\Upsilon(\ell) \ne 0$.)   

By \eqref{E:Upsilon.is.smooth.imbedding}, $\Upsilon$ is injective. It turns out that
    \begin{equation}  \label{E:plane-fitting.Upsilon.mu.is.injective}
      \Upsilon_{\mu} \text{ is injective.}
    \end{equation}  
For suppose not. Then there exist $\ell_{1}, \ell_{2} \in P^{1}$ unequal but which nonetheless satisfy $\Upsilon_{\mu}(\ell_{2}) = \Upsilon_{\mu}(\ell_{1})$. 
Let $Y_{i} := \Upsilon(\ell_{i})$ ($i=1,2$). Then, by \eqref{E:R.mu.defn}, 
$\mu \bigl( \| Y_{1} \|^{-1} Y_{1} \bigr) \| Y_{1} \|^{-1} Y_{1}$ \linebreak 
$= \mu \bigl( \| Y_{2} \|^{-1} Y_{2} \bigr) \| Y_{2} \|^{-1} Y_{2}$. I.e., for some 
$s \in \RR \setminus \{0\}$, we have $Y_{2} = s Y_{1}$. Hence, by \eqref{E:lambda(ell)=Delta.Upsilon(ell)}, 
$\lambda(\ell_{1}) = \rho(Y_{1}) = \rho(Y_{2}) = \lambda(\ell_{2})$. 
But $\lambda$ is injective by \eqref{E:lambda(ell).imbedding}. 
\eqref{E:plane-fitting.Upsilon.mu.is.injective} follows.

In fact, we have 
   \begin{lemma}   \label{L:Upsilon.mu:P1.to.Pf.is.imbedding}
If $\mu : S^{n \nvar-1} \to (0, \infty)$ is smooth then the map 
$\Upsilon_{\mu} : P^{1} \to \Pf$ is a smooth imbedding.
   \end{lemma}
For proof see appendix \ref{Chptr:misc.proofs}. By lemma \ref{L:Pf.is.a.manif}, it follows that 
    \begin{equation}  \label{E:Oops!.mu.is.imbedding}
      \Upsilon_{\mu} \text{ is also a smooth imbedding of $P^{1}$ into } \Y .
    \end{equation}

Recall \eqref{E:Upsilon.maps.P1.into.Pf}, \eqref{E:R.mu(Pf)=Pf}, and \eqref{E:Upsilon.Delta.mu.defn}. Let 
	\begin{equation}  \label{E:plane.fittin.T.mu.defn}
		\T_{\mu} := \Upsilon_{\mu}(P^{1}) \subset \Pf_{\mu} .   
	\end{equation}
The elements of $\T_{\mu}$ are $n \times \nvar$ matrices. By \eqref{E:Oops!.mu.is.imbedding}, 
Boothby \cite[Theorem (5.5), p.\ 78]{wmB75}, and \eqref{E:Pf.mu.is.manifold}, 
    \begin{multline}  \label{E:in.plane.fitting.T.mu.is.smooth.submanif}
      \text{If $\mu : S^{n \nvar-1} \to (0, +\infty)$ is smooth, } \\
        \text{then $\T_{\mu}$ is a compact smooth imbedded submanifold 
          of $\Pf_{\mu}$, and hence, of } \D_{\mu} .
    \end{multline}.

Let $w^{n \times 1}$ satisfy $w^{T} 1^{n} = 1$. By \eqref{E:when.is.Y.in.Pk}, for every 
$Y \in \Pf$ the matrix $Y - 1^{n} w^{T} Y$ has rank $k$. Recall \eqref{E:Delta(Y).defn}. By 
\eqref{E:Pf.mu.is.manifold} and \eqref{E:Delta.smooth.on.Pf}, 
	\begin{equation} \label{E:Delta.cont.on.Pf}
		Y \mapsto \Delta_{\mu}(Y) =  \rho(Y - 1^{n} w^{T} Y) \in G(k,\nvar) 
		        \text{ is continuous in } Y \in \Pf_{\mu}. 
	\end{equation}  

Let $w^{n \times 1}$ be an arbitrary column vector s.t.\  
$w^{T} 1^{n} = 1$. Then, by \eqref{E:Delta(Y).defn}, \eqref{E:R.mu.defn} and  \eqref{E:Upsilon.Delta.mu.defn}, 
    \begin{align}  \label{E:Delta.mu=Dellta.mu.circ.R.mu} 
      \Delta_{\mu} \circ R_{\mu}(Y) &= \Delta \circ R_{\mu}(Y) \notag \\ 
        &=  \rho \bigl( R_{\mu}(Y) - 1^{n} w^{T} R_{\mu}(Y) \bigr) \notag \\
        &=  \rho \Bigl( \mu \bigl( \| Y \|^{-1} Y \bigr) \| Y \|^{-1} Y 
          -  \mu \bigl( \| Y \|^{-1} Y \bigr) \| Y \|^{-1} 1^{n} w^{T} Y \Bigr) \\
        &=  \rho \Bigl( \mu \bigl( \| Y \|^{-1} Y \bigr) \| Y \|^{-1} 
          [Y - 1^{n} w^{T} Y ] \Bigr) \notag \\
        &=  \rho (Y - 1^{n} w^{T} Y ) \notag \\
        &= \Delta(Y) , \quad Y \in \Pf . \notag
    \end{align}
It now follows from \eqref{E:lambda(ell)=Delta.Upsilon(ell)} and \eqref{E:Upsilon.Delta.mu.defn} that 
    \begin{equation}  \label{E:lambda(ell)=Delta.mu.Upsilon.mu(ell)}
      \Delta_{\mu} \bigl[ \Upsilon_{\mu}(\ell) \bigr] 
        = \Delta_{\mu} \circ R_{\mu} \circ \Upsilon(\ell) 
          = \Delta \circ \Upsilon(\ell) = \lambda(\ell), \quad \ell \in P^{1}.
    \end{equation}  

Suppose $(\Phi, \Y')$ satisfies \eqref{E:plane-fitter.defn} and let $\Phi_{\mu}$ be the restriction of $\Phi$ to $\D_{\mu}$:
    \begin{equation}  \label{E:plane.fitting.Phi.mu}
      \Phi_{\mu} := \Phi \restriction_{\D_{\mu}} \text{ when defined.}
    \end{equation} 
Suppose further
    \begin{equation}  \label{E:plane.fit.D'.mu}
      \D'_{\mu} := \Y' \cap \D_{\mu} \text{ is dense in } \D_{\mu} 
    \end{equation} 
and $\Phi_{\mu} : \D_{\mu}' \to G(k,\nvar)$ is continuous. 
Let 
    \begin{equation}  \label{E:plane.fit.S.mu.defn}
      \Ss := \D_{\mu} \setminus \D'_{\mu}
    \end{equation}

\emph{Now let $\D$ be either $\D_{\mu}$ (\eqref{E:plane.fitting.D.mu.defn}) 
or $\D_{\infty}$ (\eqref{E:plane-fitting.D.infty}).} Correspondingly, let 
$\T$ be either $\T_{\mu}$ (\eqref{E:plane.fittin.T.mu.defn}) or $\T_{\infty}$ (\eqref{E:T.infty:=PS[Upsilon(P1)]}), let $\Ss$ be either $\Ss_{\mu}$ 
(\eqref{E:plane.fit.S.mu.defn}) or $\Ss_{\infty}$ (\eqref{E:plane.fit.S.infty.defn}). $\Ss$ need not be closed. Let $\D'$ be either $\D'_{\mu}$ (\eqref{E:plane.fit.D'.mu}) or $\D'_{\infty}$ (\eqref{E:plane.fit.D'.infty}). Let $\Pf = \Pf_{\mu}$ (see \eqref{E:Pf.mu.defn})  
or $\Pf = \Pf_{\infty}$ (see \eqref{E:Pf.infty.Delta.defns}). Let $\Upsilon$ be either $\Upsilon_{\mu}$ (\eqref{E:Upsilon.Delta.mu.defn}) 
or $\Upsilon_{\infty}$ (\eqref{E:Upsilon.infty.defn}), and let $\Delta$ 
be either $\Delta_{\mu}$ (\eqref{E:Upsilon.Delta.mu.defn}) or $\Delta_{\infty}$  (\eqref{E:Pf.infty.Delta.defns}). Thus, by \eqref{E:plane.fittin.T.mu.defn} or \eqref{E:T.infty:=PS[Upsilon(P1)]}, we have 
$\T = \Upsilon(P^{1})$. 
Let $\Phi$ be either $\Phi_{\mu}$ (\eqref{E:plane.fitting.Phi.mu}) or $\Phi_{\infty}$ (\eqref{E:plane.fitting.Phi.infty}). Assume, however we arrived at it, 
$(\Phi, \D, \D', \T)$ is a plane-fitter by the criteria in \eqref{E:general.plane-fitter.defn}.

We apply the ``severity trick'' (remark \ref{R:severity.trick}) with $\text{``$\Pf$''} = \T$. By \eqref{E:feature.space.in.plane.fitting}, in plane-fitting the codomain,$\F$, is the  Grassmannian, $G(k,\nvar)$, a Riemannian manifold. Therefore, by proposition \ref{P:smooth.manifs.have.convex.combos}, there exists a, possibly noncommutative, convex combination function, $\gamma$, on the cover, $\msf{V}$, 
of $\F$ consisting of all geodesically convex subsets. Using proposition \ref{P:smooth.manifs.have.commutative.convex.combos} one gets a commutative convex combination function with a finer cover. In section \ref{SS:lin.combo.for.plane.fit.w/.k=nvar-1} we construct a coarse cover with commutative convex combination function for the case 
$k = \nvar - 1$. In any case, let $\msf{V}$ be an open cover of of the codomain $\F$ with a convex combination function $\gamma$. 
Let $\Ss^{\msf{V}} \subset \D_{\mu}$ be the set of $\msf{V}$-severe singularities 
of $\Phi$. By \eqref{E:SV.is.closed}, $\Ss^{\msf{V}}$ is closed. 

By \eqref{E:plane.fittin.T.mu.defn} and 
\eqref{E:Oops!.mu.is.imbedding} and \eqref{E:plane.fittin.T.mu.defn}, or \eqref{E:Upsilon.is.smooth.imbedding} and \eqref{E:T.infty:=PS[Upsilon(P1)]}, 
$\T$ is an imbedded 1-dimensional submanifold of $\D$. Therefore, by \eqref{E:r=1}, \textbf{hypotheses \ref{Hyp:T.manif} and \ref{Hyp:r.integer}} of theorem \ref{T:Phi.star.Hr.contains.Theta.star.Hr} holds with $r = 1$, by \eqref{E:r=1}. By \eqref{E:D.mu.diffeo.to.sphere} or \eqref{E:plane.fit.D.infty.is.a.sphere}, $\D$ is diffeomorphic to a sphere, i.e., a $C^{\infty}$ manifold. Hence, by remark \ref{R:retraction.in.manifs}, $\T$ has a neighborhood $\Rcl$ with a smooth retraction 
$\Rcl \to \T$.

Let $\tilde{\D} = \D \setminus \Ss^{\msf{V}}$. (See \eqref{E:S^V.notation}.) Since we have assumed $\D'$ is dense in $\D$, $\tilde{\D} \supset \D'$ is automatically dense in $\D$. Since, by \eqref{E:SV.is.closed}, $\Ss^{\msf{V}}$ is closed, we conclude that \textbf{hypothesis \ref{Hyp:S.cap.T.closed}} of theorem \ref{T:Phi.star.Hr.contains.Theta.star.Hr} holds for $\Ss^{\msf{V}}$. 

By (\ref{E:general.plane-fitter.defn}e), $\D' \cap \T$ is dense in $\D \cap \T$. 
By (\ref{E:general.plane-fitter.defn}f), \eqref{E:Delta.mu=Dellta.mu.circ.R.mu},  \eqref{E:Delta.smooth.on.Pf}, and \eqref{E:Delta.infty.is.cont}, $\Phi \restriction_{\T}$ has a continuous extension $\Theta$ (\emph{viz.} $\Delta$) to all of $\Pf$, and hence to to all 
of $\T$. We may apply theorem \ref{T:if.lin.combo.on.F.then.can.rstrct.to.bad.sings} to 
$(\D, \D', \Phi)$ to conclude there is a continuous map 
$\Omega : \tilde{D} \to \F$ s.t.\ $\Theta = \Delta$ is a continuous extension of the restriction $\Omega \restriction_{\Pf \cap \tilde{D}}$ of $\Omega$ 
to $\Pf \cap \tilde{D}$. 

Since $r = t = 1$, by \eqref{E:0.dim.Haus.measure}, if \textbf{hypothesis \ref{Hyp:S.cap.T.small}} of theorem \ref{T:Phi.star.Hr.contains.Theta.star.Hr}
 holds for $(\Omega, \D, \Ss^{\msf{V}})$ then 
    \begin{equation*}
      \Ss^{\msf{V}} \cap \T = \varnothing \; \footnote{This fails for least squares regression (LS; section \ref{SS:lin.reg.and.LS}, example \ref{R:dim.S.for LS}), but not for PC (section \ref{SS:PC.plane.fitting}) or least absolute deviation regression (section \ref{SS:LAD}). (LAD requires a special choice 
of $\mbf{Y}$. See section \ref{SSS:codim.LAD.sings}.) Let $\Ss'$ be a closed superset of a singular set $\Ss$ in plane-fitting. In section \ref{SS:gnrl.lwr.bnd.plane.fit} we derive another bound on $\text{codim} \, \Ss'$ in plane-fitting that holds even if \textbf{hypothesis \ref{Hyp:S.cap.T.small}} fails. That bound holds for LS.}
    \end{equation*}

Let $\Theta$ be the restriction $\Delta \restriction_{\T}$. Then, \textbf{hypothesis \ref{Hyp:extend}} of theorem \ref{T:Phi.star.Hr.contains.Theta.star.Hr} holds. Then, by \eqref{E:general.plane-fitter.defn}, with the possible exception of \textbf{hypothesis \ref{Hyp:S.cap.T.small}}, all the hypotheses of theorem \ref{T:Phi.star.Hr.contains.Theta.star.Hr} hold 
with $(\Omega, \Ss' = \Ss^{\msf{V}})$ in place of $(\Phi, \Ss')$. 

By \eqref{E:plane.fittin.T.mu.defn} and lemma \ref{L:Upsilon.mu:P1.to.Pf.is.imbedding}, or \eqref{E:T.infty:=PS[Upsilon(P1)]} and \eqref{E:Upsilon.is.smooth.imbedding} we have that 
$\Upsilon : P^{1} \to \T$ is a homeomorphism. Therefore, by \eqref{E:lambda(ell)=Delta.mu.Upsilon.mu(ell)}, or 
\eqref{E:lambda(ell)=Delta.infty.Upsilon.infty(ell)}, and  
	\begin{multline}  \label{E:lambda(ell)=Theta.Upsilon(ell)}
		\Theta \bigl[ \Upsilon(\ell) \bigr] = \lambda(\ell), 
		  \text{ for every } \ell \in P^{1} . \\
		  \text{ Equivalently, } 
		    \Theta(x) = \lambda \bigl[  \Upsilon^{-1}(x) \bigr], 
		      \text{ for every } x \in \T .
	\end{multline}  
By \eqref{E:lambda.ast.nontriv}, $\lambda_{\ast}$ is a non-trivial homomorphism of homology in dimension 1. Therefore, applying the homology functor to the second of the preceding equations we 
see that $\Theta_{\ast}$ is nontrivial in dimension $r=1$. I.e., 
    \begin{equation}  \label{E:Theta.ast.nontriv}
      \eqref{E:nontriv.r-dim.homol} \text{ holds for } \Theta . 
    \end{equation}
I.e., \eqref{E:nontriv.r-dim.homol} holds in both the ``$\infty$'' and ``$\mu$'' contexts.

By \eqref{E:D.mu.diffeo.to.sphere} or \eqref{E:plane.fit.D.infty.is.a.sphere}, 
$\D$ is, or is at least diffeomorphic to, a round sphere of dimension $d \geq 5$.  
Since $r = 1$, by \eqref{E:r=1}, by Munkres \cite[Corollary 47.2, p.\ 281]{jrM84}
we have 
    \begin{equation}  \label{E:H.d-r.D.=0.plane.fitting}
      H^{d-r}(\D) = \{0\} 
    \end{equation}
with any coefficient ring, in particular with $\ZZ/2$ coefficients.

Thus, we can apply proposition \ref{P:sing.dim.when.H.d-r.D.=.0} to $\Omega$ and conclude the following.
   \begin{prop} \label{P:sing.codim.in.plane-fitting}
Suppose $\D$ is the $(n \nvar - 1)$-sphere $\D_{\mu}$ or the $n \nvar$-sphere 
$\D_{\infty}$. Suppose $\Phi : \D \partlyto G(k, \nvar)$ is a plane-fitter as defined in \eqref{E:general.plane-fitter.defn}
and let $\Ss$ be its singular set. There is an open cover 
of $G(k, \nvar)$ on which is defined a convex combination function. 
Let $\msf{V}$ be such a cover. Assume that $\Ss' = \Ss^{\msf{V}}$ satisfies 
\textbf{hypothesis \ref{Hyp:S.cap.T.small}} of theorem \ref{T:Phi.star.Hr.contains.Theta.star.Hr}.
Then $\Hm^{d-2}(\Ss) \geq \Hm^{d-2}(\Ss^{\msf{V}}) > 0$, where $d := \dim \D$. Thus, 
    \begin{equation} \label{E:big.lwr.bound.in.plane.fit}
       codim \, \Ss \leq codim \, \Ss^{\msf{V}} \leq 2.
    \end{equation}
   \end{prop} 
 
  \begin{remark}[Sales pitch in plane-fitting]  \label{R:sales.pitch.in.plane.fitting}
Plane-fitting provides a stunning example of the ``sales pitch'' (remark \ref{R:sales.pitch}). Suppose one is considering a method for fitting a million dimensional plane to a data set consisting of a trillion points in billion dimensional space. One might learn something about the global stability properties of the method by examining its behavior in the immediate vicinity of a topological circle, a compact 1-dimensonal set! Circles like the $\T_{\mu}$ or $\T_{\infty}$ defined in this section work for all examples described in this chapter. (But some might work better than others, as in section \ref{SSS:codim.LAD.sings}. The boundary of the triangle in figure \ref{F:LF.plot.pipeline} functions as a $\T$, but it is constructed differently than the $\T$ in this section.)
  \end{remark}
 
  \begin{remark}
Not every choice of $\mbf{Y} \in \Y$ of full rank for which there exists $w^{n \times 1}$ satisfying \eqref{E:plane.fitting.wT.Y=0} will lead to a $\T$ s.t.\ 
$\Ss' = \Ss^{\msf{V}}$ satisfies 
\textbf{hypothesis \ref{Hyp:S.cap.T.small}} of theorem \ref{T:Phi.star.Hr.contains.Theta.star.Hr}. In section \ref{SSS:codim.LAD.sings}, for example, we must make a special choice of $\mbf{Y}$. 
  \end{remark}

When proposition \ref{P:sing.dim.when.H.d-r.D.=.0} applies (e.g., when proposition \ref{P:sing.codim.in.plane-fitting} applies) to $(\Omega, \D')$, we have, by example \ref{Ex:basic.property.case}, 
that $(\Omega, \Ss^{\msf{V}}, G, \T, a)$ satisfies property \ref{Pty:agree.near.T} with $G$ the trivial group and, by \eqref{E:dim.D.mu.in.plane.fit} and \eqref{E:plane.fit.D.infty.is.a.sphere}, resp., 
    \begin{align}  \label{E:a.mu.a.infty.defns}
         a &:= a_{\mu} := d_{\mu}-2 := \dim \D_{\mu} - 2 = (n \nvar - 1) - 2 = n \nvar - 3 
          \text{ for $\D = \D_{\mu}$, or } \\
         a &:= a_{\infty} := d_{\infty}-2 := \dim \D_{\infty} - 2 = n \nvar -2 
          \text{ for } \D = \D_{\infty} \notag .
    \end{align}
(Recall that, by \eqref{E:r=1}, in this chapter $r = 1$.) 
   
Put on $\D$ the Riemannian metric induced by the ambient Euclidean space, 
$\RR^{n \nvar}$ in the $\D = \D_{\mu}$ case and $\RR^{n \nvar + 1}$ in the 
$\D = \D_{\infty}$ case. Therefore, by lemma \ref{L:biLip.triangulation}, it has a bi-Lipschitz triangulation (in fact, by the boundary of a simplex of dimension $d+1$). 

$\Pf$ is an imbedded submanifold of $\D$, by \eqref{E:Pf.mu.is.manifold} in the 
$\Pf_{\mu}$ case or \eqref{E:Pf.infty.is.submanif}
in the $\Pf_{\infty}$ case. Hence, by example \ref{Ex:TubeNbhdSpecialCaseOfCones}, 
$\Pf$ has a neighborhood in $T \D \restriction_{\Pf}$, the tangent space of $\D$ restricted to $\Pf$, fibered over $\Pf$ by cones in the sense of definition \ref{D:fibering.by.cones}.  
 
It follows from this that, if proposition \ref{P:sing.codim.in.plane-fitting} holds 
for $\Phi$ (so $(\Omega, \Ss^{\msf{V}}, G, \T, a)$ satisfies property \ref{Pty:agree.near.T})  holds for $(\Omega, \Ss^{\msf{V}}, \T, \Pf, a)$ with $a = d-2$.  But by remark \ref{R:degenerate.data} \emph{we still cannot apply theorem \ref{T:lwr.bnd.on.Haus.meas} to plane-fitting!}

  \begin{remark}[``Degenerate'' data sets]  \label{R:degenerate.data}
 Say that $Y \in \Y$ is ``degenerate" if the following holds.
    \begin{equation} \label{E:when.Y.is.degenerate}
       \text{If } w^{n \times 1} \text{ satisfies } w^{T} 1^{n} = 1 
       \text{ then } rank \, ( Y - 1^{n} w^{T} Y ) < k .
    \end{equation}
For example, if $y^{1 \times \nvar} \in \RR^{\nvar}$ then $1^{n} y$ is degenerate. 
Let $\Y_{degen}$ denote the set of degenerate data sets.

For $Y \in \Y \setminus \{0\}$, define $Y^{0}$ and the ``sample covariance  matrix'', $cov(Y)^{\nvar \times \nvar}$, of $Y$ as in \eqref{E:sample.covariance.matrix}. Let 
$\lambda_{1}(Y) \geq \ldots \geq \lambda_{\nvar}(Y) \geq 0$ be the eigenvalues of $cov(Y)$. Then $Y$ is degenerate if and only if $\lambda_{k} = 0$. 

Let $Y \in \Y_{degen}$ be degenerate. Then by definition, for some $\ell < k$ we have 
$rank \, ( Y - 1^{n} w^{T} Y ) = \ell$. By \eqref{E:when.is.Y.in.Pk} and \eqref{E:Perfect.Fits.in.plane.fitting}, $Y$ lies exactly on a plane of dimension $\ell < k$.  
Therefore, $Y$ lies on infinitely many planes of dimension $k$. An arbitrarily small perturbation in $Y$ can yield a data set lying exactly one of those infinitely $k$-dimensional planes. Therefore, if $\Phi$ is a plane-fitter then, by part 3 of \eqref{E:plane-fitter.defn}, degenerate data sets are singularities of $\Phi$. In fact, they are probably severe: $\Y_{degen} \subset \Ss^{\msf{V}}$. (See section \ref{SS:lin.combo.for.plane.fit.w/.k=nvar-1}.)

Suppose $\ell = 0$, so $Y = 1^{n} w^{T} Y$. Call the set of all $Y$ for which this is true $\Pf^{0}$. Let $Y \in \Pf^{0}$ and recall \eqref{E:1n.col.vec.defn}.  
Now let $w^{n \times 1}$ satisfy $w^{T} 1^{n} = 1$. We have
    \begin{align*}
      Y - 1^{n} w^{T} Y &= Y - 1^{n} w^{T} 
        \bigl[ Y - n^{-1}1^{n} 1_{n} Y +  n^{-1}1^{n} 1_{n} Y \bigr] \\
          &= Y - 1^{n} w^{T} 
           \bigl[ (Y - n^{-1}1^{n} 1_{n} Y) + n^{-1}1^{n} 1_{n} Y \bigr] \\
              &= Y - n^{-1}1^{n} (w^{T} 1^{n}) 1_{n} Y) \\
                &= Y - n^{-1}1^{n} 1_{n} Y = 0 .
    \end{align*}
Thus, $Y \in \Pf^{0}$ if and only if $0 = (I_{n} - n^{-1}1^{n} 1_{n}) Y$. The diagonal elements of $I_{n} - n^{-1}1^{n} 1_{n}$ are $(n-1)/n$ and the off diagonal elements are $-n^{-1}$. Thus, $(I_{n} - n^{-1}1^{n} 1_{n}) Y = 0$ if and only if the columns of $Y$ are constant, i.e., for some $b^{1 \times \nvar}$ we have $Y = 1^{n} b$. So 
$\Pf^{0} = \{ 1^{n} b \in \Y : b \text{ is } 1 \times \nvar \}$. Therefore 
$\dim \Pf^{0} = \nvar$. Thus, lemma \ref{L:Pf.is.a.manif} holds with $k = 0$. 

For $\ell > 0$, $Y$ is a perfect fit, just one corresponding to a smaller $k$. We conclude that $\Y_{degen} = \bigcup_{\ell=0}^{\ell=k-1} \Pf^{\ell}$. Hence, by lemma \ref{L:Pf.is.a.manif} (and \eqref{E:dim.of.whole.=.max.dim.of.parts}), 
$\dim \Y_{degen} = \dim \Pf^{k-1}$. So 
    \begin{equation*}  
      \text{codim} \, \Y_{degen} = n \nvar - \bigl[ n(k-1) + k(\nvar - k+1) \bigr]
        = n \nvar - n(k-1) - k(\nvar-k+1) .
    \end{equation*} 
Recalling \eqref{E:n>nvar>k>0}, we see that this expression is decreasing in $k \in (0,  \nvar)$. Hence, its minimum value is 
    \begin{equation}  \label{E:codim.of.degenerates}
      \text{codim} \, \Y_{degen} \geq 2(n-\nvar+1) \geq 4 .
    \end{equation}
Comparing this to \eqref{E:big.lwr.bound.in.plane.fit}, we see that practically none of the singularities of a $\Phi$ satisfying the hypotheses of proposition \ref{P:sing.codim.in.plane-fitting} are degenerate. 

The argument given in section \ref{SS:PC.plane.fitting} to prove \eqref{E:PC.Y'.dense} serves to show that 
    \begin{equation*}
      \Y_{degen} \subset \overline{\Pf} ,
    \end{equation*}
where  is the closure of $\Pf$. Let $Y \in \Y_{degen}$. 

In the notation of section \ref{SS:PC.plane.fitting}, the eigenvalues of $cov(Y)$ are $\nu_{1}^{2}/(n-1) \geq \cdots \geq \nu_{\nvar}^{2}/(n-1)$. Since $Y$ is degenerate, $\nu_{k}^{2} = \cdots = \nu_{\nvar}^{2} = 0$. For $\epsilon > 0$, the covariance matrix of the data set $Y_{\epsilon}$ has $k^{th}$ eigenvalue $> 0$, but all subsequent eigenvalues are 0. Thus, $Y_{\epsilon} \in \Pf$. But as $\epsilon \downarrow 0$, $Y_{\epsilon} \to Y$. This proves that the degenerate data sets lie in the closure of $\Pf$. I.e., 
    \begin{equation*}
      dist^{a}(\Ss, \Pf) = 0 ,
    \end{equation*}
where $a$ is defined in \eqref{E:a.mu.a.infty.defns}. This means, \emph{prima facie}, we cannot apply theorem \ref{T:lwr.bnd.on.Haus.meas} to plane-fitting, probably not even with $\Ss' = \Ss^{\msf{V}}$. 

We can get around this by strengthening the notion of perfect fit. (See \eqref{E:Perfect.Fits.in.plane.fitting}.) Recall that if $Y \in \Y$ and 
$\lambda_{1}(Y) \geq \ldots \geq \lambda_{\nvar}(Y) \geq 0$ are the eigenvalues of 
$cov(Y)$ then $\lambda_{k} = 0$. 

Let $c > 0$. By lemma \ref{L:Eigen.cont.}, the function $\lambda_{k}(Y)/(\lambda_{1}(Y) + c)$ is continuous on $\Y$. Let $r \in (0,1)$ and define 
$\Pf_{r,c} := \{ Y \in \Pf : \lambda_{k}(Y)/(\lambda_{1}(Y) + c) > r \}$.   
$\lambda_{1}(Y) \geq \lambda_{k}(Y) > r \lambda_{1}(Y) + r c > 0$. 
Recall \eqref{E:set.distances} and that $Euc$ is the Euclidean metric on $\Y$: $Euc(Y,Y') := \|Y-Y'\|$ ($Y,Y' \in \Y$). \emph{Claim:}  
    \begin{equation}  \label{E:dist(Pf.rc,Y.degen)}
      dist_{Euc}(\Pf_{r,c}, \Y_{degen}) > rc > 0 .
    \end{equation}
(The metrics on $\D_{\infty}$ and $\D_{\mu}$ are not the same as $Euc$. In fact, 
on $\D_{\infty}$, the point at infinity lies in the closure of both $\Pf_{r,c}$ and 
$\Y_{degen}$ so we have $dist(\Pf_{r,c}, \Y_{degen}) = 0$. So theorem \ref{T:lwr.bnd.on.Haus.meas} cannot be used on $\D_{\infty}$. However, by \eqref{E:metrics.on.D.mu}, in the $\D_{\mu}$ setting, $dist_{Euc}(\Pf_{r,c}, \Y_{degen})$ will be bounded below by a multiple of $rc$.) To prove \eqref{E:dist(Pf.rc,Y.degen)}, let $Y \in \Pf_{r,c}$. 
Let $\bar{y}^{1 \times \nvar} := n^{-1} 1_{n} Y$ be the vector of column means of $Y$. 
By the Singular Value Decomposition (Rao \cite[(v), p.\ 42]{crR73.LinStatInf}), we may write 
$Y = 1^{n} \bar{y} + L^{n \times \nvar} \Lambda N^{T}$, where $L$ has orthonormal columns, $\Lambda^{\nvar \times \nvar}$ is non-negative diagonal, and $N^{\nvar \times \nvar}$ is orthogonal. We may assume the elements along the diagonal of $\Lambda$ are non-increasing from the upper left to the lower right. Since $Y \in \Pf_{r,c} \subset \Pf$, by \eqref{E:when.is.Y.in.Pk}, the last $\nvar - k$ diagonal elements of $\Lambda$ are 0. Let $M^{\nvar \times \nvar}$ be the diagonal matrix obtained from $\Lambda$ by replacing the $k^{th}$ diagonal element of $\Lambda$ by 0.  

Let $Y' := 1^{n} \bar{y} + L^{n \times \nvar} M N^{T}$. Then 
$Y' \in \Y_{degen}$. By the Eckhart-Young theorem (Eckart and Young \cite{cEgY36.eckhart-young}), $Y'$ is the closest element of 
$\Y_{degen}$ to $Y$. By \eqref{E:switch.and.trace}, we have
    \begin{multline*}
      \| Y' - Y \|^{2} = trace \, (Y' - Y) (Y '- Y)^{T}
        = trace \, \bigl[ L (M-\Lambda) N^{T} \bigr] \bigl[ L (M-\Lambda) N^{T} \bigr]^{T} \\
          = trace \, L (M-\Lambda)^{2} L^{T} 
            = L^{T} L (M-\Lambda)^{2} = trace \, (M-\Lambda)^{2} .
    \end{multline*}
The $k^{th}$ diagonal entry in $M - \Lambda$ is $r \lambda_{1}(Y) + r c  > rc > 0$. Hence, 
$\| Y' - Y \| > rc > 0$. This completes the proof of the claim \eqref{E:dist(Pf.rc,Y.degen)}.
  \end{remark}
  
A data set in $\Pf_{r,c}$ robustly determines a $k$-plane and using 
``$\Pf\text{''} = \Pf_{r,c}$ instead of $\Pf$ as defined in \eqref{E:Perfect.Fits.in.plane.fitting} makes applying theorem \ref{T:lwr.bnd.on.Haus.meas} to plane-fitting a possibility. 
Define $\Pf_{\mu}$ and $\Pf_{\infty}$ based on $\Pf_{r,c}$ instead of $\Pf$. Generically, denote either one by $\Pf$. By \eqref{E:Pf.mu.is.manifold}, 
$p_{\mu} := \dim \Pf_{\mu} = nk + (k + 1)(\nvar - k) - 1$. By \eqref{E:Pf.infty.is.submanif}, 
$p_{\infty} := nk + (k + 1)(\nvar - k)$. By \eqref{E:dim.D.mu.in.plane.fit}, \eqref{E:a.mu.a.infty.defns}, and \eqref{E:n>nvar>k>0}, we then have 
    \begin{multline*}
      d_{\mu} - p_{\mu} -1 = (nq - 1) - \bigl[ nk + (k + 1)(\nvar - k) - 1 \bigr] - 1 \\
        = n \nvar - k (n+\nvar-k) - (\nvar -k) - 1 < n \nvar - 3 = a_{\mu} .
    \end{multline*}
Similarly, using \eqref{E:plane.fit.D.infty.is.a.sphere}, we have 
    \begin{multline*}
      d_{\infty} - p_{\infty} -1 = nq - \bigl[ nk + (k + 1)(\nvar - k) \bigr] - 1 \\
        = n \nvar - k (n+\nvar-k) - (\nvar -k) - 1 < n \nvar - 2 = a_{\infty} .
    \end{multline*}

Let $\Hm_{\D}^{a}$ denote $a$-dimensional Hausdorff measure on $\D$. Suppose $dist^{a}(\Ss, \Pf) \geq R > 0$ (see \eqref{E:essential.dist.defn}) and 
$dist^{a}(\Ss^{\msf{V}}, \Pf) \geq R^{\msf{V}} > 0$. So  
$R^{\msf{V}} \geq R$. Then by \eqref{E:Hm.a.S.geq.R.d-p-1},  
         \begin{equation}  \label{E:Hm.a.S.lwr.bnd.for.plane.fitting.Pf}
            \Hm_{\D}^{a}(\Ss) \geq \Hm_{\D}^{a}(\Ss^{\msf{V}}) 
              \geq \gamma \; (R^{\msf{V}})^{ d-p-1 }
                \geq \gamma R^{ d-p-1 } ,
         \end{equation}
where $\gamma > 0$ does not depend on $(\Phi, \Ss)$ or $(\Omega, \Ss^{\msf{V}})$. 
(As mentioned above, at least without further adjustments, theorem \ref{T:lwr.bnd.on.Haus.meas} cannot be applied to $\D_{\infty}$.)

   \begin{remark}[Different measures on $\D$]  \label{R:different.measures.on.D} 
$\Hm_{\D}^{a}$ appearing in \eqref{E:Hm.a.S.lwr.bnd.for.plane.fitting.Pf} is computed using the geodesic-based metric on $\D$. (As assumed in theorem \ref{T:lwr.bnd.on.Haus.meas}.) Recall that $Euc$ is the Euclidean metric on $\Y$. Let $\delta$ be a metric on $\D_{\mu}$ satisfying \eqref{E:metrics.on.D.mu}. (Recall that $\eta$ is the Riemannian metric on $\D_{\mu}$ induced by the inclusion $\D_{\mu} \hookrightarrow \RR^{n \nvar}$. The corresponding topological metric was called $Eta$. Then $\delta \geq Euc$ and, as observed above, $\delta = Eta$ satisfies \eqref{E:metrics.on.D.mu}.  Let $\Hm_{\Y}^{a}$ be computed using $Euc$ on $\Y$. 
Therefore, by \eqref{E:Lip.magnification.of.Hm}, there exists $K < \infty$ s.t.\ 
$\Hm_{\D_{\mu}}^{a} \leq K^{a} \Hm_{\Y}^{a}$. Thus, \eqref{E:Hm.a.S.lwr.bnd.for.plane.fitting.Pf} holds with $\Hm_{\D}^{a}(\Ss)$ 
(where $\D = \D_{\mu}$) replaced by $\Hm_{\Y}^{a}(\Ss)$, where \emph{$\Ss$ is still the singular set of the restriction $\Phi \restriction_{\D_{\mu}}$}.

But there is another wrinkle. Let $\Y'$ be a dense subset of $\Y$ satisfying 
\eqref{E:plane-fitter.defn}. Suppose $\D' := \Y' \cap \D_{\mu}$ satisfies 
\eqref{E:general.plane-fitter.defn}. Singularities of the restriction $\Phi \restriction_{\D_{\mu}}$ relative to $\D'$ will also be singularities of $\Phi$ 
in $\D_{\mu}$ relative to $\Y'$. But the converse is false. There may be $\Y'$ singularities in $\D_{\mu}$ that are not $\D'$ singularities. (This phenomenon is similar to that described in remark \ref{R:Phi.on.and.near.T}. The same point is also made in remark \ref{R:restricting.plane.fitter.to.sphere}.) Thus, the alternate version of \eqref{E:Hm.a.S.lwr.bnd.for.plane.fitting.Pf}, the one with $\Hm_{\D}^{a}$ replaced by $\Hm_{\Y}^{a}$, will still hold with $\Ss$ replaced by the set of singularities on $\D_{\mu}$ relative to $\Y' \cap \D_{\mu}$. 

As for $\D = \D_{\infty}$, let $\Ss$ be the singular set for a $\Phi$ defined on a dense subset, $\Y'$, of \emph{all of} $\Y$. (See the paragraph preceding \eqref{E:plane.fit.S.infty.defn}.) In real life $\Hm_{\Y}^{a}(\Ss)$ will be infinite.
   \end{remark}

  \begin{remark}[Shifting $\D_{\mu}$] \label{R:shift.D.mu}
Suppose $\Y' \subset \Y$ is dense, $\Phi : \Y' \to G(k,\nvar)$, and one wishes to localize and consider the behavior of $\Phi$ in some region of $\Y$. One way to do this is to shift $\D_{\mu}$ to that region and consider the restriction of $\Phi$ on the shifted $\D_{\mu}$. (See remark \ref{R:restricting.plane.fitter.to.sphere}). For example given $y_{0}^{1 \times \nvar} \in \RR^{\nvar}$ one might consider 
$\D_{\mu}(y_{0}) := \D_{\mu} + 1^{n} y_{0} \subset \Y$.  

Many plane-fitters $\Phi$ are translation invariant: For every 
$Y \in \Y$ for which $\Phi(Y)$ is defined, $\Phi(Y + 1^{n} y_{0})$ is defined and constant in $y_{0} \in \RR^{\nvar}$. (See \eqref{E:shft.invar.of.regrssn} for a special case.) For such 
a $\Phi$, any $\D_{\mu}$ as in \eqref{E:plane.fitting.D.mu.defn}, and $y \in \RR^{\nvar}$, we can learn about the singularities of $\Phi$ on $\D_{\mu}(y_{0})$ by applying our theory 
to $\D_{\mu}$.

In general, if $\D_{\mu}$ is shifted to $\D_{\mu}(y_{0})$ it is natural to shift 
$\T_{\mu}$ to $\T_{\mu}(y_{0}) := \T_{\mu} + 1^{n} y_{0}$. 
\textbf{Hypotheses \ref{Hyp:T.manif}, and \ref{Hyp:r.integer}} of theorem \ref{T:Phi.star.Hr.contains.Theta.star.Hr} continue to hold with 
$\T = \T_{\mu}(y_{0})$ and $\D = \D_{\mu}(y_{0})$. Suppose 
$\D' \cap \D_{\mu}(y_{0})$ is dense in $\D_{\mu}(y_{0})$ and let 
$\Ss'(y_{0}) \subset \D_{\mu}(y_{0})$ be a closed superset 
of $\D_{\mu}(y_{0}) \setminus \D'$. 

Even if \textbf{hypotheses \ref{Hyp:S.cap.T.closed} and \ref{Hyp:S.cap.T.small}} of theorem \ref{T:Phi.star.Hr.contains.Theta.star.Hr} hold with $\T = \T_{\mu}$ and 
$\D = \D_{\mu}$, they may not hold with $\T = \T_{\mu}(y_{0})$ and 
$\D = \D_{\mu}(y_{0})$. If \textbf{hypothesis \ref{Hyp:S.cap.T.closed}} fails then $\Ss'(y_{0})$ has non-empty interior in $\D_{\mu}(y_{0})$. 
That means $\dim \bigl( \Ss'(y_{0}) \bigr) = \dim \D_{\mu}(y_{0}) = d-1$. 
If \textbf{hypothesis \ref{Hyp:S.cap.T.small}} fails 
$\Ss'(y_{0}) \cap \T_{\mu}(y_{0}) \neq \varnothing$. 

By \eqref{E:Pf.invar.under.shifts.rescale}, $\Pf$ is shift invariant in $\Y$. Hence, by 
\eqref{E:Pf.mu.defn}, $\Pf_{\mu}  + 1^{n} y_{0} \subset \Pf$. Similarly, by 
\eqref{E:plane.fittin.T.mu.defn}, we have $\T_{\mu}(y_{0}) \subset \Pf$. Hence, 
$\Delta$ (see \eqref{E:Delta(Y).defn}) is defined and continuous on 
$\T_{\mu}(y_{0})$.  

Let $Y \in \T_{\mu}$ and let $w^{n \times 1}$ be as in \eqref{E:Delta(Y).defn}: 
$w^{T} 1^{n} = 1$. Then 
    \begin{align*}
      \Delta(Y + 1^{n} y_{0}) 
        &= \rho \bigl[ (Y + 1^{n} y_{0}) - 1^{n} w^{T} (Y + 1^{n} y_{0}) \bigr] \\
        &= \rho \bigl[ Y + 1^{n} y_{0} - 1^{n} w^{T} Y 
            - 1^{n} (w^{T} 1^{n}) y_{0}) \bigr] \\
        &= \rho \bigl[ Y + 1^{n} y_{0} - 1^{n} w^{T} Y - 1^{n} y_{0}) \bigr] \\
        &= \rho( Y - 1^{n} w^{T} Y ) .
    \end{align*}
I.e., $\Delta(Y + 1^{n} y_{0}) = \Delta(Y)$. Hence, as before, \eqref{E:nontriv.r-dim.homol} also holds with $\Theta := \Delta$. 

Note that $\D_{\mu}(y_{0})$ is homeomorphic to $\D_{\mu}$. Therefore, by 
\eqref{E:H.d-r.D.=0.plane.fitting}, the hypothesis ``$\D$ is a compact $d$-dimensional manifold with $\check{H}^{d-r}(\D) \isomto H^{d-r}(\D) = \{ 0 \}$'' in proposition \ref{P:sing.dim.when.H.d-r.D.=.0} holds with $\D = \D_{\mu}(y_{0})$, $r = 1$ (and, e.g., 
$\ZZ/2$ coefficients). Therefore, if $\Ss'(y_{0})$ is closed, by proposition 
\ref{P:sing.dim.when.H.d-r.D.=.0}, \eqref{E:codim.S.leq.r+1} holds with 
$d = \dim \D_{\mu}(y_{0}) = \dim \D_{\mu} = n \nvar - 1$ and $r = 1$. 
So $\text{codim} \, \Ss'(y_{0}) \leq 2$.
  \end{remark}
  
    \begin{remark} \label{R:restricting.plane.fitter.to.sphere}
Suppose $\Phi : \Y' \to G(k, \nvar)$, where $\Y'$ is dense in $\Y$. It still might make sense to consider the behavior of $\Phi$ on a subspace $\D_{\mu} \subset \Y$. An advantage of doing so is that it allows one to locate somewhat where the singularities are. This idea is applied in remark \ref{R:dnsity.contrs.as.data.spaces}. By shifting $\D_{\mu}$ (remark \ref{R:shift.D.mu}) one gains flexibility in doing this. Call the act of restricting a plane-fitter to some $\D_{\mu}$ ``localization''.

This is only sensible if $\D' \cap \D_{\mu}$ is dense in $\D_{\mu}$. (In theory, 
$\D' \cap \D_{\mu} = \varnothing$ is possible.) If $\D' \cap \D_{\mu}$ is dense we may consider the restriction, $\Phi_{\mu} := \Phi \restriction_{ \D' \cap \D_{\mu}}$ of $\Phi$ 
to $\D' \cap \D_{\mu}$. 

By supposition, $\Y'$ is dense in $\Y$. Suppose $\Y'$ is invariant under rescaling. 
I.e., if $s > 0$, then $s \Y' = \Y'$. \emph{Claim:} $\Y' \cap \D_{\mu}$ is dense 
in $\D_{\mu}$. Let $\Ss_{\mu} := \D_{\mu} \setminus \Y'$. If $\Ss_{\mu}$ is empty, then $\Y' \cap \D_{\mu} = \D_{\mu}$ and we are done. 

So assume $\Y' \cap \D_{\mu}$ is not dense in $\D_{\mu}$. It follows that $\Ss_{\mu}$ has non-empty interior, $\mcl{O}_{\mu} := \Ss_{\mu}^{\circ}$, in $\D_{\mu}$. Therefore, by \eqref{E:D.mu.diffeo.to.sphere}, $\mcl{O}_{\mu}^{1} := (R_{\mu} \restriction_{S^{n \nvar - 1}})^{-1}(\mcl{O}_{\mu})$ is open in $S^{n \nvar - 1}$. But, by \eqref{E:R.mu.defn}, $R_{\mu}$ is a pointwise rescaling. Hence, by \eqref{E:g.commutes.w/.set.ops} and scale invariance of $\Y'$, \linebreak 
$\mcl{O}_{\mu}^{1} \subset (R_{\mu} \restriction_{S^{n \nvar - 1}})^{-1}
(\D_{\mu} \setminus \Y') = S^{n \nvar - 1} \setminus \Y'$. Thus, by scale invariance of $\Y'$ again, 
    \begin{equation*}
      (0,+\infty) \mcl{O}_{\mu}^{1} \subset 
        (0,+\infty) S^{n \nvar - 1} \setminus (0,+\infty) \Y' 
          = (0,+\infty) S^{n \nvar - 1} \setminus \Y' .
    \end{equation*} 
This means 
    \begin{equation}  \label{E:(0,infty) S1.disjoint.from.Y'}
      (0,+\infty) \mcl{O}_{\mu}^{1} \cap \Y' = \varnothing .
    \end{equation}
    
Let $g(Y) := \|Y\|^{-1} Y \in S^{n \nvar - 1}$ ($Y \neq 0$), so $g$ is continuous. 
Then $(0,+\infty) \mcl{O}_{\mu}^{1} = g^{-1}(\mcl{O}_{\mu}^{1})$. This means 
$(0,+\infty) \mcl{O}_{\mu}^{1} \neq \varnothing$ is open in $\Y$. But, by 
\eqref{E:(0,infty) S1.disjoint.from.Y'}, this means $\Y'$ is disjoint from an open subset of $\Y$. This contradicts the supposition that $\Y'$ is dense in $\Y$. The claim that 
$\Y' \cap \D_{\mu}$ is dense in $\D_{\mu}$. is proved. 

Even if $\Y'$ is invariant under rescaling it is possible that $\Phi_{\mu}$ can be extended to a subset, $\D_{\mu}'$, of $\D_{\mu}$ larger than $\Y' \cap \D_{\mu}$. (See remark \ref{R:different.measures.on.D}.) In that case apply lemma \ref{L:extend.Phi.to.D.less.S} so that \eqref{E:D'.=.D.less.S} holds on $\D_{\mu}$. Otherwise, let 
$\D_{\mu}' := \Y' \cap \D_{\mu}$.  
  \end{remark}

\subsection{``Loose'' plane-fitting}  \label{SS:loose.plane.fitting}
Suppose $\Phi : \Y \partlyto \F$. Here we see that 
(\ref{E:general.plane-fitter.defn}e,f) can be replaced by a weaker requirement. Recall \eqref{E:Delta(Y).defn} and \eqref{E:defn.of.Pi(xi)}. Then we replace (\ref{E:general.plane-fitter.defn}e,f)can be replaced by,
    \begin{multline}  \label{E:loose.plane.fit.defn}
      \Y' \cap \T \text{ is dense in } \T, \, \\
        \Phi : \Y' \to G(k,\nvar) \text{ and } \Phi \restriction_{\T \cap \Y'} 
          \text{ has a continuous extension $\Omega$ to all of $\T$ s.t.\  } \\
            \text{ if } x \in \T \text{ then for every } v \in \Omega(x) \setminus \{0\} 
              \text{ we have } v \Pi_{\Delta(x)} \neq 0 .
    \end{multline}

Suppose \eqref{E:loose.plane.fit.defn} holds. Let $x \in \T$, 
$v \in \Omega(x) \setminus \{0\}$, and $\lambda \in [0,1]$. Recall \eqref{E:superscript.perp.notation}. Then 
$I_{\nvar} - \Pi_{\Delta(x)} = \Pi_{\Delta(x)^{\perp}}$. Hence, 
    \begin{equation*}
      \lambda v + (1-\lambda) v \Pi_{\Delta(x)} 
        = v \bigl[ \lambda (I_{\nvar} - \Pi_{\Delta(x)}) + \Pi_{\Delta(x)} \bigr]
          = v (\lambda \Pi_{\Delta(x)^{\perp}} + \Pi_{\Delta(x)}) .
    \end{equation*}
 Suppose $\lambda > 0$. If 
 $v (\lambda \Pi_{\Delta(x)^{\perp}} + \Pi_{\Delta(x)}) = 0$ then the orthogonal projections of $v$ onto the two perpendicular planes whose direct sum if the whole space $\RR^{\nvar}$ is 0. That cannot happen because $v \neq 0$. If $\lambda = 0$, then, by \eqref{E:loose.plane.fit.defn}, it cannot be 0. 
 
Let 
    \begin{equation*}
      \xi(\lambda,x) := \lambda \Omega(x) + (1-\lambda) \Omega(x)  \Pi_{\Delta(x)} 
        := \bigl\{ \lambda v + (1-\lambda) v \Pi_{\Delta(x)} \in \RR^{\nvar} : 
          v \in \Omega(x) \bigr\}
    \end{equation*}
It follows from the preceding paragraph that $\dim \xi(\lambda,x) \geq k$. 
Since $\dim \Omega(x) = k$, we have $\dim \xi(\lambda,x) \leq k$. 
I.e. $\xi(\lambda,x) \in G(k,\nvar)$. Thus, $\xi$ is a homotopy between $\Delta$ and 
$\Omega$. 

By \eqref{E:Delta(Y).defn} and lemma \ref{L:proj.mat.is.imbedding.of.Grass}, 
$\xi(\lambda,x)$ is continuous in $\lambda \in [0,1]$ and $x \in \T \cap \Y'$. Suppose 
$\T \subset \Y$ is a manifold s.t.\ \eqref{E:nontriv.r-dim.homol} holds with 
$\Theta = \Delta$ and $r = 1$. By \eqref{E:Theta.ast.nontriv.in.dim.1} this is true 
if $\T$ is as in section \ref{SS:D.T.plane.fit}. Since $\Delta$ and the $\Omega$ are homotopic, we have \eqref{E:Theta.ast.nontriv.in.dim.1} holds for the 
$\Theta = \Omega$. 

Now the whole of section \ref{SS:D.T.plane.fit} goes through for $\Phi$ satisfying \eqref{E:loose.plane.fit.defn}.  

\section{Principal components plane fitting} \label{SS:PC.plane.fitting} 
Principal components plane fitting (PC) is defined in example \ref{Ex:3.plane.fitters}. 
Let $\D' = \Y' \subset \Y$, be the set of data sets $Y$ for which 
$\Phi_{PC} := PC(Y)$ is defined. Let $Y \in \Y'$ and let 
$\lambda_{1} \geq \cdots \lambda_{\nvar} \geq 0$ be the eigenvalues of $cov(Y)$. By definition of $\Y'$, $\lambda_{k} > \lambda_{k+1}$. And conversely, if $Y \in \Y$ and the eigenvalues of $cov(Y)$ are $\lambda_{1} \geq \cdots \lambda_{\nvar} \geq 0$ with 
$\lambda_{k} > \lambda_{k+1}$, then $Y \in \Y'$. By lemma \ref{L:Eigen.cont.}, there is an open neighborhood of $Y$ in which the $k^{th}$ eigenvalue of the covariance matrix is strictly bigger than the $(k + 1)^{st}$. This proves that $\Y'$ is open. 

We prove that 
    \begin{equation}  \label{E:Pf.subset.Y'.for.PC}
      \Pf \subset \Y' .
    \end{equation} 
Let $Y \in \Pf$ and let 
$w^{n \times 1} = n^{-1} 1^{n}$. Then the matrix $Y^{0}$, defined in example \ref{Ex:3.plane.fitters}, is just $Y - 1^{n} w^{T} Y$ and therefore, by \eqref{E:when.is.Y.in.Pk}, has rank $k$. By the Singular Value Decomposition (Rao \cite[(v), p.\ 42]{crR73.LinStatInf}), we may write 
    \begin{equation}  \label{E:Y0.SVD}
      Y^{0} = L^{n \times \nvar} \Lambda N^{T} ,
    \end{equation}
where $L$ has orthonormal columns, $\Lambda^{\nvar \times \nvar}$ is non-negative diagonal, and $N^{\nvar \times \nvar}$ is orthogonal. By \eqref{E:sample.covariance.matrix}, 
    \begin{equation}  \label{E:cov(Y).SVD}
      cov(Y) =  \frac{1}{n-1}(Y^{0})^{T} \, Y^{0} = \frac{1}{n-1} N \Lambda^{2} N^{T} .
    \end{equation}
It follows that the covariance matrix $cov(Y)$ (see \eqref{E:sample.covariance.matrix}) also has rank $k$. Hence, if $\lambda_{1} \geq \cdots \lambda_{\nvar} \geq 0$ are the eigenvalues of $cov(Y)$, then $\lambda_{k} > 0 = \lambda_{k+1}$. 
Hence, $\Phi_{PC} \in G(k, \nvar)$ is defined. 

Still let $Y \in \Pf$. As we just observed, $rank(Y^{0}) = k = rank(cov(Y))$. Therefore, all but $k$ of the elements of $\Lambda$ are non-zero. Moreover, the diagonal elements of 
$\Lambda^{2}$ are  the eigenvalues of $cov(Y)$ and columns of $N$ are eigenvectors. 
Recall \eqref{E:rho=row.space}. Then $\Phi_{PC}(Y) = \rho(cov(Y)) = \rho(Y^{0})$. 
 But by \eqref{E:Delta(Y).defn}, $\Delta(Y) = \rho(Y - 1^{n} \bar{y}) = \rho(Y^{0})$. 
I.e., 
    \begin{equation}  \label{E:Phi.PC=Delta.on.Pf}
      \Delta(Y) = \Phi_{PC}(Y) \text{ if } Y \in \Pf .
    \end{equation} 
Thus, $\Phi_{PC}$ satisfies (\ref{E:general.plane-fitter.defn}f). 

\emph{Claim:} 
    \begin{equation}  \label{E:PC.Y'.dense}
      \Y' \text{  is dense in } Y .
    \end{equation} 
Let $Y \in \Y \setminus \Y'$. Then, if 
$\lambda_{1} \geq \cdots \lambda_{\nvar} \geq 0$ are the eigenvalues of $cov(Y)$, then $\lambda_{k} = \lambda_{k+1}$. As in example \ref{Ex:3.plane.fitters}, write 
$(Y^{0})^{n \times \nvar} := Y - 1^{n} \bar{y}$, where 
$\bar{y}^{1 \times \nvar} := n^{-1} 1_{n} Y$. Write $Y^{0}$ as in \eqref{E:Y0.SVD}. 

Then, by \eqref{E:cov(Y).SVD}, the eigenvalues of $cov(Y)$ are $(n-1)^{-1} \lambda_{1}^{2} \geq \cdots (n-1)^{-1} \lambda_{k}^{2} = (n-1)^{-1} \lambda_{k+1}^{2} \geq \cdots 
\geq (n-1)^{-1} \lambda_{\nvar}^{2}$ with $\lambda_{k}^{2} = \lambda_{k+1}^{2}$. And the columns of $N$ are corresponding unit eigenvectors.
 
Recall \eqref{E:1n.col.vec.defn}. \emph{Claim:} We may assume $1_{n} L = 0$. Since $1_{n} Y^{0} = 0$, we have $0 = (1^{n} L \Lambda N^{T}) N = L \Lambda$. Thus, every column of $L$ corresponding to a non-zero column (i.e., non-zero diagonal entry) 
of $\Lambda$ is perpendicular to $1^{n}$. Thus, if there are no zero diagonal entries 
in $\Lambda$ then $1_{n} L = 0$. Suppose there are $j = 1, \ldots, \nvar$ zero diagonal entries in $\Lambda$. Then $L$ has $\nvar - j$ orthogonal columns also orthogonal to $1^{n}$. Those columns span a $(\nvar - j)$-dimensional subspace, $V$, 
of the $(n-1)$-dimensional space $(1^{n})^{\perp}$. (See \eqref{E:superscript.perp.notation}.) Replace the other $j$ columns of $L$ by $j$ orthogonal vectors in $(1^{n})^{\perp}$ that are also perpendicular to $V$. Then we still have $Y^{0} = L^{n \times \nvar} \Lambda N^{T}$, but also $(1_{n})^{T} L = 0$. This proves the claim that we may assume $(1_{n})^{T} L = 0$. 
 
Let $\nu_{1}, \ldots, \nu_{\nvar}$ be the diagonal elements of $\Lambda$. WLOG we may assume $\nu_{1} \geq \cdots \geq \nu_{\nvar} \geq 0$. 
Thus, the eigenvalues of $(n-1) \, cov(Y)$ are 
$\nu_{1}^{2} \geq \cdots \geq \nu_{\nvar}^{2}$. Since $Y \in \Y \setminus \Y'$ by assumption, we must have $\nu_{k} = \nu_{k+1}$. Let $\epsilon > 0$ and let 
$\mu_{i} = \nu_{i} + \epsilon/i$ for $i = 1, \ldots, k$. Thus, $\mu_{j} > \mu_{j+1}$ 
for $j = 1, \ldots, k$. Let $M^{\nvar \times \nvar}$ be diagonal with diagonal entries 
$\mu_{1} > \cdots \mu_{k} > \mu_{k+1} > \cdots > \mu_{\nvar}$ and let 
$Y_{\epsilon} := L^{n \times \nvar} M N^{T} + 1^{n} \bar{y}$. Since $1_{n} L = 0$, 
we have $\bar{y}_{\epsilon}^{1 \times \nvar} := n^{-1} 1_{n} Y_{\epsilon} 
= \bar{y}^{1 \times \nvar} := n^{-1} 1_{n} Y$. 
Let $Y_{\epsilon}^{0} = Y_{\epsilon} - 1^{n} \bar{y}$. 

As in \eqref{E:cov(Y).SVD},  
    \begin{equation}  \label{E:cov(Y.eps).SVD}
      cov(Y_{\epsilon}) =  \frac{1}{n-1}(Y_{\epsilon}^{0})^{T} \, Y_{\epsilon}^{0} 
        = \frac{1}{n-1} N M^{2} N^{T} .
    \end{equation}
Hence, the eigenvalues of $cov(Y_{\epsilon})$ are 
$(n-1)^{-1} \mu_{1}^{2} > \cdots (n-1)^{-1} \mu_{k}^{2} > (n-1)^{-1} \mu_{k+1}^{2} 
> \cdots > (n-1)^{-1} \mu_{\nvar}^{2}$. And the columns of $N$ are corresponding unit eigenvectors. Thus $Y_{\epsilon} \in \Y'$. By making 
$\epsilon > 0$ small, $Y_{\epsilon}$ can be made as close to $Y$ as desired. This proves the claim \eqref{E:PC.Y'.dense} that $\Y'$ is dense in $\Y$. 

It follows from the preceding that $Y \in \Y \setminus \Y'$ is a singularity. Let $v_{j}$  denote the $j^{th}$ column of $N$. 
Since $\lambda_{k}^{2} = \lambda_{k+1}^{2}$, \eqref{E:Y0.SVD} still holds if $v_{k}$ and $v_{k+1}$ are replaced by $v_{k,\theta} := (\cos \theta) v_{k} + (\sin \theta) v_{k+1}$ and 
$v_{k+1,\theta} := (\sin \theta) v_{k} - (\cos \theta) v_{k+1}$, resp., for any $\theta \in \RR$. Then, by \eqref{E:cov(Y.eps).SVD}, the PC plane of $Y_{\epsilon}$ is spanned 
by $v_{1}, \ldots, v_{k-1}, v_{k,\theta}$. 
As $\epsilon \downarrow 0$, $Y_{\epsilon} \to Y$, but the PC plane remains constant. Allowing $\theta$ to vary, we therefore get different planes in the limit. I.e., $Y$ is a singularity.

Next, we show \emph{claim:} 
    \begin{equation}  \label{E:Phi.PC.is.cont}
      \Phi_{PC} \text{ is continuous on } \Y' .
    \end{equation} 
Let $Y \in \Y'$. Since $\Y'$ is open, $Y$ has a neighborhood 
$\mcl{V} \subset \Y'$. Let $\{ Y_{m} \} \subset \mcl{V}$ (so $\Phi_{PC}(Y_{m})$ is defined) s.t.\ $Y_{m} \to Y$ (in $\| \cdot \|$ norm of course). 

Let $\lambda_{m,1} \geq \cdots \geq \lambda_{m,\nvar}$ be the eigenvalues 
of $(n-1) \, cov(Y_{m})$. Let $\ell = 1, \ldots, \nvar$. By lemma \ref{L:Eigen.cont.} again, 
$\lambda_{m,\ell} \to \lambda_{\ell}$, the $\ell^{th}$ eigenvalue of $(n-1) \, cov(Y)$. 
 
Let $Q_{m}^{\nvar \times \nvar}$ be a matrix whose rows are eigenvectors 
of $(n-1) \, cov(Y_{m})$. We may assume that the $\ell^{th}$ row of $Q_{m}$ is a unit (row) eigenvector, $v_{m,\ell}$, with eigenvalue $\lambda_{m,\ell}$. Taking a subsequence if necessary, the lemma also tells us that $Q_{m}$ converges to a matrix whose rows are the eigenvectors of $(n-1) \, cov(Y)$. In particular, if 
$\ell \neq \ell'$ then $v_{m,\ell}$ and $v_{m,\ell'}$ converge to distinct eigenvectors 
of $(n-1) \, cov(Y_{m})$. 

Let $\delta := \lambda_{k} - \lambda_{k+1}$. Since $Y \in \Y'$, we have $\delta > 0$ and $\lambda_{\ell} - \lambda_{\ell'} \geq \delta$ for $\ell \leq k < \ell'$. 
Clearly, $cov(Y)$ is continuous in $Y \in \Y$. Therefore, $Y_{m} \to Y$ means 
$\| cov(Y_{m}) - cov(Y) \| \to 0$. Thus, 
    \begin{align*} 
        | \lambda_{m,\ell} v_{m,\ell} - \lambda_{\ell'} v_{\ell'} | 
        &= (n-1) \bigl| v_{m,\ell} \, cov(Y_{m}) - v_{\ell'} \, cov(Y) \bigr| \\
        &= (n-1) \bigl| \bigl( v_{m,\ell} \, cov(Y) - v_{\ell'} \, cov(Y) \bigr) 
           + (n-1) \bigl( v_{m,\ell} \, cov(Y_{m}) - v_{m,\ell} \, cov(Y) \bigr) \bigr| \\
        &\leq (n-1) | v_{m,\ell} - v_{\ell'} | \| cov(Y) \| 
          + (n-1) \| cov(Y_{m}) - cov(Y) \| \to 0 . 
    \end{align*}
On the other hand, since $v_{m,\ell}$ and $v_{\ell'}$ are unit vectors and 
$v_{m,\ell} \to v_{\ell'}$, we have, 
    \begin{multline*}
      | \lambda_{m,\ell} v_{m,\ell} - \lambda_{\ell'} v_{\ell'} |^{2}
        = \lambda_{m,\ell}^{2} 
          - 2 \lambda_{m,\ell} \lambda_{\ell'} (v_{m,\ell} \cdot v_{\ell'})
            + \lambda_{\ell'}^{2} \\ 
              \to \lambda_{\ell}^{2} - 2 \lambda_{\ell} \lambda_{\ell'}
               + \lambda_{\ell'}^{2} = ( \lambda_{\ell} - \lambda_{\ell'} )^{2}
                 \geq \delta^{2} > 0 .
    \end{multline*}
Contradiction. 
Thus, if $\ell \leq k$ then $v_{m,\ell} \to v_{\ell'}$, where $v_{\ell'}$ is a unit eigenvector 
of $(n-1) \, cov(Y)$ with eigenvalue $\lambda_{\ell'}$ and $\ell' \leq k$. Similarly, 
if $\ell > k$ then $v_{m,\ell} \to v_{\ell'}$ with $\ell' > k$. 

Let $V_{m}^{k \times \nvar}$ be the matrix whose rows are $v_{m,\ell}$ with 
$\ell = 1, \ldots, k$. Then $V_{m}$ converges in $\| \cdot \|$ norm to a matrix $V^{k \times \nvar}$ whose $\ell^{th}$ row is a unit eigenvector of $(n-1) \, cov(Y)$ with eigenvalue $\lambda_{\ell}^{2}$ and $\ell \leq k$. And the limits constitute all the first $k$ eigenvectors of $(n-1) cov(Y)$. Then, by \eqref{E:convergence.in.Grassmann}, 
$\Phi_{PC}(Y_{m}) = \rho(V_{m}) \to \rho(V) = \Phi_{PC}(Y)$. This proves the claim \eqref{E:Phi.PC.is.cont}. 

Note that $\D'$ and $\Phi_{PC}$ are scale and shift invariant: If $Y \in \Y'$, $s > 0$, 
and $y^{1 \times k} \in \RR^{k}$, 
then $cov(s Y + 1^{n} y) = s \, cov(Y)$. Thus, $s Y + 1^{n} y \in \Y'$ and 
$\Phi_{PC}(s Y + 1^{n} y) = \Phi_{PC}(Y)$. Recall \eqref{E:Pf.subset.Y'.for.PC} and \eqref{E:Phi.PC=Delta.on.Pf}. Then, whether 
    \begin{itemize}
        \item $\D = \Y$, $\D' = \Y'$, and $\T$ as in \eqref{E:plane.fittin.T.defn}; OR
        \item $\D = \D_{\infty} = PS(\Y')$, $\D' = \D_{\infty}' := PS(\Y')$, 
          and $\T = \T_{\infty}$ as in \eqref{E:T.infty:=PS[Upsilon(P1)]} 
            (and see \eqref{E:plane.fittin.T.mu.defn}); OR 
        \item $\D = \D_{\mu}$ as in \eqref{E:plane.fitting.D.mu.defn}, 
        $\D' = \D'_{\mu} := \Y' \cap \D_{\mu}$, and $\T = \T_{\mu}$ as in \eqref{E:plane.fittin.T.mu.defn} (and see \eqref{E:in.plane.fitting.T.mu.is.smooth.submanif}) for some smooth $\mu$ as in \eqref{E:mu.as.in.D.mu.defn};
    \end{itemize}
in all those cases we have that $\Phi_{PC}$ satisfies the requirements spelled out in \eqref{E:general.plane-fitter.defn} and $\Phi_{PC}$ is continuous on $\T$. Hence, proposition \ref{P:sing.codim.in.plane-fitting} and \eqref{E:Hm.a.S.lwr.bnd.for.plane.fitting.Pf} applies to PC.

Since $\D'$ and $\Phi_{PC}$ are scale and shift invariant, the ideas in remarks  \ref{R:restricting.plane.fitter.to.sphere} and \ref{R:dnsity.contrs.as.data.spaces} apply to PC as well. It turns out that, as exemplified in figure \ref{F:sing.dists.cdfs}, the codimension of the singular set of $\Phi_{PC}$ is exactly 2
(\cite[Proposition 1.5, p.\ 6]{Ellis.dim.sing.long.version}). More generally, the singular sets of certain``projection pursuit plane fitting'' methods, 
\cite[Example 2.4, pp.\ 494--496]{spE95}, have codimension no bigger than 2.)
PC is often used as a version of ``factor analysis''. As such it has in addition another, more subtle, form of singularity (\cite[Examples 1.1 and 7.1]{spE.fact.anal}).)

\section{Linear regression in general, least squares linear regression in particular}  \label{SS:lin.reg.and.LS}
Linear regression is one of the most common statistical activities and by far the most common form it takes is least squares. The data take the form 
    \begin{equation} \label{E:initially.in.lin.reg.nvar=k+1}
      Y = (X^{n \times k}, y^{n \times 1}), \text{ so } \nvar = k+1 .
    \end{equation}
The columns of $X$ are the ``predictors'' or ``independent variables'' and $y$ is the column vector of ``responses''. The vector $y$ is also called the ``dependent variable''.
Let $x_{i}^{1 \times k}$ be the $i^{th}$ of $X$ and $y_{i} \in \RR$ the $i^{th}$ entry in $y$ ($i=1, \ldots, n$). In linear regression the following notion is important. 
    \begin{definition}  \label{D:collinearity}
$Y = (X^{n \times k}, y^{n \times 1}) \in \Y$ is ``(multi)collinear'' 
if $x_{2} - x_{1}, \ldots, x_{n}- x_{1}$ do not span $\RR^{k}$.
    \end{definition}
 (Notice that this definition still makes sense if we relax \eqref{E:n>nvar>k>0} and allow 
 $n = \nvar$.)
 
Let $Y = (X, Z)$ be collinear. It follows from lemma \ref{L:rank.lwr.semicont} 
that by making arbitrarily small perturbations in $X$, we can get a noncollinear data set. Thus,
    \begin{equation}  \label{E:noncollinear.dense.inY}
      \text{The set of all non-collinear data sets is dense in } \Y.
    \end{equation}
This is made precise in lemma \ref{L:dim.set.of.collin.data.sets}
 
Let $f_{1}^{n \times 1} := (1, 0, \ldots, 0)^{T}$, so $f_{1}^{T} 1^{n} = 1$.. Then $Y$ is collinear if $rank \, (X - 1^{n} w^{T} X) < k$, with $w = f_{1}$. There is nothing special about this choice of $w$. By \eqref{E:const.rowspace.of.offset.mats}, 
    \begin{multline}  \label{E:nothing.special.about.x1.in.collin}
      \text{Let $w^{n \times 1}$ satisfy $w^{T} 1^{n} = 1$. Then } \\
        (X^{n \times k}, y^{n \times 1}) \text{ is collinear if and only if }
          rank \, (X - 1^{n} w^{T} X) < k .
    \end{multline}
(Lemma \ref{L:collinearity.and.w.1xn} generalizes this.) 

Let $1^{n}$ be the $n$-dimensional column vector consisting only of $1$'s. In linear regression 
plane-fitting one seeks $a \in \RR$ and a column vector $b^{k \times 1}$ s.t.\ 
$a 1^{n} + X b \in \RR^{n}$ approximates $y$ well in some sense. Call the pair $(a, b^{T})$ the ``regression of $y$ on $X$''. The corresponding plane, 
    \begin{equation}   \label{E:in.lin.reg.Phi.is.graph}
      \Phi(X,y) \text{ is the graph } \bigl\{ (x, x^{1 \times k} \, b^{k \times 1}) : 
        x \in \RR^{k} \text{ is a row vector} \bigr\},
    \end{equation}
which is the $k$-dimensional subspace parallel to the ``regression plane'': the graph of the affine function (``linear regression function'')
    \begin{equation}  \label{E:linear.regression.function}
      f : x^{1 \times k} \mapsto a 1^{n} + x b .
    \end{equation} 
($x^{1 \times k} \, b^{k \times 1})$ is, of course, just the inner product of $x$ with $b^{T}$.) Given $x \in \RR^{k}$, but not the corresponding $y \in \RR$, the regression function $f$ is a data map that can be used to ``predict'' $y$. In this section, we are not interested in $f$. Rather we examine the process of ``learning'' $f$ from the data $Y$. (See remark \ref{R:learning.and.predicting}.)

Anyway, 
    \begin{multline}  \label{E:Phi(Y).orthog.to.in.span.of}
      \Phi(Y) \text{ is the orthogonal complement of } (b^{T}, -1) \text{ in } \RR^{\nvar}, \\
        \text{ the same as the row space of } (I_{k}, b), 
    \end{multline}
where $I_{k}$ is the $k \times k$ identity matrix. Note that $\Phi(Y)$ is also the orthogonal complement of $\pm \bigl| (b^{T}, -1) \bigr|^{-1} (b^{T}, -1) \in S^{k}$. Note that $b$ can be recovered from $\pm \bigl| (b^{T}, -1) \bigr|^{-1} (b^{T}, -1)$. 
 
Often, linear regression is shift invariant. This means that for 
$v^{1 \times k} \in \RR^{k}$ and $c \in \RR$ arbitrary, 
	\begin{multline}  \label{E:shft.invar.of.regrssn} 
	      \text{If the coefficient vector for the regression of $y$ on $X$ is } ( a, b ) . \\
	         \text{ Then the coefficient vector for the regression of } 
                  y - c 1^{n} \text{ on } X - 1^{n} v \\
                    \text{ is } ( a - c + v b, b ). 
	\end{multline}
I.e., changing the data in this way has no impact on $b$. A possible choice 
of $v^{1 \times k}$ 
might be $x_{i}$ for some $i = 1, \ldots, n$. (\eqref{E:shft.invar.of.regrssn} may fail for Bayesian or ``shrinkage'' methods Berger \cite{joB85.BergerBayes}, 
Hoerl and Kennard \cite{aeHrwK70.ridge.reg}, and Tibshirani \cite{rT96.lasso}.)

  \begin{remark}  \label{R:linear.reg.of.colinear.data}
Here we consider, in the regression setting, perfect fits (i.e., data in $\Pf$) that are almost collinear. We show, that as a noncollinear perfect fit approaches a collinear perfect fit, its coefficient vector goes to infinity. To see this, suppose we are employing a regression method satisfying \eqref{E:shft.invar.of.regrssn} and suppose 
    \begin{equation}  \label{E:Y.is.perfect.and.collin}
      Y^{n \times \nvar} = (X, y) \in \Pf \text{ is collinear.}
    \end{equation} 
Let $x_{1}^{1 \times k}$ be the first row of $X$ and let $X_{0} := X - 1^{n} x_{1}$. Let 
    \begin{equation*}
      w^{n \times 1} = (1, 0, \ldots, 0)^{T} .
    \end{equation*} 
Thus, 
    \begin{equation}  \label{E:X0.w}
      X_{0} = X - 1^{n} x_{1} = X - 1^{n} w^{T} X, \text{ so } w^{T} X_{0} = 0 .
    \end{equation} 
Since $Y$ is collinear, by definition \ref{D:collinearity}, $rank \, X_{0} < k$. 
  
Let $y_{1} \in \RR$ be the first entry in $y$ and let 
    \begin{equation}  \label{E:y0.defn}
      y_{0}^{n \times 1} := y - y_{1} 1^{n} = y - 1^{n} w^{T} y \; 
        \text{ so } \; w^{T} y_{0} = 0 .
    \end{equation} 
By \eqref{E:Y.is.perfect.and.collin} and \eqref{E:when.is.Y.in.Pk}, 
$rank \, (X_{0}, y_{0}) = rank \, ( Y - 1^{n} w^{T} Y ) = k$. But $rank \, X_{0} < k$. It follows that:
    \begin{equation}  \label{E:rank.X0=k-1}
      rank \, X_{0} = k-1 \text{ and } y_{0} 
        \text{ is not in the column space of } X_{0} . 
    \end{equation}
Define 
    \begin{equation}  \label{E:y'=proj.of.y0.onto.col.space}
      (y')^{n \times 1} = \text{ orthogonal projection of } y_{0} 
        \text{ onto the column space of }X_{0} .
    \end{equation}
Therefore, there exists $(b')^{k \times 1}$ s.t.\
    \begin{equation}  \label{E:y'=X0.b'}
      y' = X_{0} b' . 
    \end{equation}
We may assume 
    \begin{equation}  \label{E:b.orthog.to.b'}
      (b')^{T} \in \rho(X_{0}) .
    \end{equation} 
(See \eqref{E:rho=row.space}.)

By \eqref{E:y'=X0.b'} and \eqref{E:X0.w}, $w^{T} y' = 0$. Since, by \eqref{E:rank.X0=k-1}, $y_{0}$ is not in the column space of $X_{0}$, by \eqref{E:y0.defn} we have  
    \begin{equation}  \label{E:y''.perp.w.X0}
      (y'')^{n \times 1} := y_{0} - y' \neq 0 \text{ and } (y'')^{T} (w, X_{0}) = 0 .
    \end{equation}

Let 
    \begin{equation}  \label{E:b.perp.rowspace.of.X0}
      b^{k \times 1} \text{ be a unit vector s.t.\ $b^{T}$ is orthogonal to } \rho(X_{0}) .
    \end{equation}
Hence, by \eqref{E:b.orthog.to.b'}, 
    \begin{equation}  \label{E:b.perp.b'}
      b \text{ is orthogonal to } b' .
    \end{equation}
(Since $rank \, X_{0}^{n \times k} = k-1$ by \eqref{E:rank.X0=k-1}, $b$ is unique up to sign.) 

Now we perturb $X$ in the direction $(y'' b^{T})^{n \times k}$. 
Let $\epsilon \in \RR \setminus \{0\}$ and let 
    \begin{equation*}
      \tilde{X}_{\epsilon} := \tilde{X} := \epsilon y'' b^{T} + X .
    \end{equation*} 
Write 
    \begin{equation}  \label{E:X0.tilde.defn}
      \tilde{X}_{0} := \tilde{X} - 1^{n} w^{T} \tilde{X} 
        \; \text{ so } \; w^{T} \tilde{X}_{0} = 0 , 
    \end{equation}
and 
    \begin{equation}  \label{E:Y.tilde.defn}
      \tilde{Y} := \tilde{Y}_{\epsilon} := (\tilde{X}_{\epsilon}, y) .
    \end{equation} 
By \eqref{E:X0.w} and \eqref{E:y''.perp.w.X0}, we have 
    \begin{equation}  \label{E:w.X.tilde=x1}
      w^{T} \tilde{X} = \epsilon w^{T} y'' b^{T} + w ^{T}1^{n} x_{1} 
        + w^{T} X_{0} = x_{1}.
    \end{equation} 
Therefore by \eqref{E:X0.w} again, 
    \begin{equation}  \label{E:X0.tilde=y'' b+ X0}
      \tilde{X}_{0} = (\epsilon y'' b^{T} + 1^{n} x_{1} + X_{0}) - 1^{n} x_{1} 
        = \epsilon y'' b^{T} + X_{0} .
    \end{equation}
Hence, by \eqref{E:y''.perp.w.X0} and \eqref{E:b.perp.rowspace.of.X0}, 
$\tilde{X}_{0} b = \epsilon y'' \neq 0$. Therefore, by \eqref{E:rank.X0=k-1},
    \begin{equation}  \label{E:rank.X0.tilde=k}
       rank \, \tilde{X}_{0} = k . 
    \end{equation}
Therefore, by definition \ref{D:collinearity}, \eqref{E:X0.tilde.defn}, and \eqref{E:rank.X0.tilde=k}, 
    \begin{equation*}
      \tilde{Y} \text{ is not collinear.}
    \end{equation*} 

By \eqref{E:y''.perp.w.X0} and \eqref{E:y'=X0.b'},  
    \begin{equation}  \label{E:y0=y''+X0.b'}
      y_{0} = y'' + y' = y'' + X_{0} b' .
    \end{equation}
    
Recall \eqref{E:Y.tilde.defn}, \eqref{E:y0.defn}, \eqref{E:X0.tilde.defn},  \eqref{E:w.X.tilde=x1}, and \eqref{E:y0.defn}. We compute the rank of 
    \begin{equation}  \label{E:Ytilde.0.defn}
      \tilde{Y}_{0} := \tilde{Y} - 1^{n} w^{T} \tilde{Y} = (\tilde{X}_{0}, y_{0}) .
    \end{equation} 

By \eqref{E:y0=y''+X0.b'} and \eqref{E:X0.tilde=y'' b+ X0}, we have
$\tilde{Y}_{0} = ( \tilde{X}_{0} , y_{0} )  
= ( \epsilon y'' b^{T} + X_{0} , y'' + X_{0} b' )$. 
By \eqref{E:b.perp.rowspace.of.X0} and \eqref{E:b.orthog.to.b'}, we have 
$b^{T} b' = 0 = X_{0} b$ and $b^{T} b = 1$. Hence, by \eqref{E:y0=y''+X0.b'}, \eqref{E:y'=X0.b'}, and  \eqref{E:X0.tilde=y'' b+ X0}, 
the last column of $\tilde{Y}_{0}$ can be re-expressed as follows.
    \begin{multline}  \label{E:y0.in.terms.of.X0.tilde}
       y_{0} = y' + y'' = X_{0} b' + y'' = (\epsilon y'' b^{T} + X_{0}) b' 
         + \epsilon^{-1} (\epsilon y'' b^{T} + X_{0}) b \\
           = \tilde{X}_{0} b' 
             + \epsilon^{-1} \tilde{X}_{0} b
               = \tilde{X}_{0} (b' + \epsilon^{-1} b) . 
    \end{multline}
Thus, the last column of $\tilde{Y}_{0}$ is in the column space of $\tilde{X}_{0}$. Therefore, by \eqref{E:rank.X0.tilde=k}, 
    \begin{equation*}
      rank \, \tilde{Y}_{0} = k .
    \end{equation*} 
Thus, by \eqref{E:when.is.Y.in.Pk} and \eqref{E:Ytilde.0.defn}, we have 
$\tilde{Y}_{\epsilon} \in \Pf$.  

Let 
    \begin{equation*}
      \tilde{b} := \tilde{b}_{\epsilon} := b' + \epsilon^{-1} b .
    \end{equation*}
Then, by \eqref{E:b.perp.rowspace.of.X0}, $\tilde{b} \to \infty$ as $\epsilon \to 0$. 
By \eqref{E:X0.tilde.defn}, \eqref{E:y0.in.terms.of.X0.tilde}, and \eqref{E:y0.defn},
    \begin{equation*}
      \tilde{X} \tilde{b} = \tilde{X}_{0} \tilde{b} + 1^{n} w^{T} \tilde{X} \tilde{b} 
        = y_{0} + (w^{T} \tilde{X} \tilde{b}) 1^{n} 
          = y + ( -w^{T} y + w^{T} \tilde{X} \tilde{b} ) 1^{n} .
    \end{equation*}
Thus, if $a := w^{T} y - w^{T} \tilde{X} \tilde{b}$, then
    \begin{equation*}
      y = a 1^{n} + \tilde{X}_{\epsilon} \tilde{b}_{\epsilon} . 
    \end{equation*}

Suppose $\Phi$ satisfies \eqref{E:plane-fitter.defn} and is defined at all non-collinear perfect fits, a not unreasonable demand to put on a regression method (example: \eqref{E:noncollin.in.Pf.is.in.Y.LAD'}). Then the regression coefficient vector for the regression of $y$ on $\tilde{X}$ is 
$\tilde{b}_{\epsilon}$, which goes to infinity as $\epsilon \to 0$, in other words 
as $\tilde{Y}_{\epsilon} \to Y$. But we may replace $b$ by $-b$. This will make 
$\tilde{b}_{\epsilon} - b'$ go to infinity in the opposite direction.

Recall \eqref{E:b.perp.rowspace.of.X0}: $b$ is a unit vector. As observed after \eqref{E:Phi(Y).orthog.to.in.span.of}, $\Phi(\tilde{X}_{\epsilon}, y)$ is the orthogonal complement of 
    \begin{equation*}
      \pm |\tilde{b}_{\epsilon}|^{-1} (\tilde{b}_{\epsilon}, -1) 
        = \pm |\epsilon b' + b|^{-1} (\epsilon b' + b, -\epsilon)  \to (\pm b, 0) ,
          \text{ as } \epsilon \to 0, 
    \end{equation*} 
i.e., as $\tilde{Y}_{\epsilon} \to Y$. 

Thus, collinear data sets in $\Pf$ can be thought of as the points at infinity in ``$b$-space'' and we may consider linear regression to be a map into the projective space $P^{k}$. But we already knew that: By Milnor and Stasheff \cite[Lemma 5.1, p.\ 57]{jwMjdS74}), 
$\F = G(k, k+1) \homeomto G(1,k+1) = P^{k}$. (See \eqref{E:feature.space.in.plane.fitting} and \eqref{E:initially.in.lin.reg.nvar=k+1}.) Do not confuse the projective space $P^{k}$ with the perfect fit space $\Pf^{k}$.) We exploit this fact in section \ref{SS:lin.combo.for.plane.fit.w/.k=nvar-1} to construct a convex combination function for linear regression.
  \end{remark} 

  \begin{remark}[Mean centering]  \label{R:mean.center}
A useful choice of $v^{1 \times k}$ in \eqref{E:shft.invar.of.regrssn} is the mean of all the rows of $X$ (i.e., the row vector consisting of all the column means): $v := n^{-1} 1_{n} X$. In that case, the operation of replacing $X$ by $X - 1^{n} v$ is called ``mean centering'' $X$. Similarly, $b$ is unaffected if we take $c$ in \eqref{E:shft.invar.of.regrssn} 
to be $c = \bar{y}$, the mean of $y$. (This is ``mean-centering'' $y$.) Thus, $b$ is unaffected if we mean center the whole matrix $Y$ by mean-centering both $X$ and $y$. By \eqref{E:nothing.special.about.x1.in.collin}, if $Y$ is collinear, then it remains collinear if we mean center $X$, $y$, or both. 
  \end{remark}

We stated that the goal of linear regression is to approximate $y$ well. In least squares linear regression (LS), approximating $y$ well means that $a$ and $b$ are chosen to minimize the $L^{2}$ norm of $y - a 1^{n} - X b$. The singular set of LS consists precisely of the collinear data sets (Ellis \cite[Example 2.8]{spE95}, proposition \ref{P:collin.data.are.sings.of.LS}).

We examine this issue in the context of a more general procedure, viz.\ multivariate least squares multiple regression (Anderson \cite[Section 8.2, pp.\ 287--289]{Anderson}). Let $k$, $m$, and $n$ be positive integers with 
	\begin{equation}  \label{E:n.>.k+m}
		n > \nvar := k + m.
	\end{equation}
The data consists of pairs $(x_{i}, z_{i})$,  ($i = 1, \ldots, n$), where $x_{i}$ is a $k$-dimensional (row) vector (the predictor) and $z_{i}$ is a $m$-dimensional (row) vector (the response). 
Let $Y$ be the matrix whose $i^{th}$ row is $(x_{i}, z_{i})$, ($i =  1, \ldots, n$). Let $X^{n \times k}$ be the matrix whose $i^{th}$ row is $x_{i}$ ($ i = 1, \ldots, n$). Let $Z^{n \times m}$ be the matrix whose $i^{th}$ row is $z_{i}$, ($i = 1, \ldots, n$). Thus, $Y = (X , Z)$. 
Let $\hat{\beta}$ be a $k \times m$ matrix and $\hat{\alpha}$ a $m$-dimensional row vector s.t.\  $\beta = \hat{\beta}$ and $\alpha = \hat{\alpha}$ minimize
	\begin{equation}  \label{E:LS.criterion}
		\sum_{i=1}^{n} | z_{i} - \alpha - x_{i} \beta |^{2}.
	  \end{equation}
(Here, $| \cdot |$ is the usual Euclidean norm.) The pair $\hat{\alpha}$ and $\hat{\beta}$ are ``least squares (LS) estimates for the regression of $Z$ on $X$ or for $Y$''. Define the corresponding ``LS plane for $Y$'' to be 
    \begin{equation}  \label{E:LS.plane}
      \text{LS plane for } Y = \bigl\{ (x, x \hat{\beta}) : 
        x^{1 \times k} \in \RR^{k} \bigr\} \in G(k, \nvar) .
    \end{equation}
(See \eqref{E:in.lin.reg.Phi.is.graph}.) If there is only one such plane,
 i.e. if $\hat{\beta}$ exists uniquely, denote the plane by $\Phi_{LS}(Y)$. Since we ignore $\hat{\alpha}$ here, in general $\Phi_{LS}(Y)$ is not the LS regression plane passing through $Y$.

Let $v^{1 \times k}$ be arbitrary. E.g., $v$ could be $n^{-1} 1_{n} X$, the mean of the rows of $X$. Then obviously, by \eqref{E:LS.criterion},
    \begin{multline} \label{E:subtracting.const.doesn't.change.LS.beta}
      \hat{\alpha} \text{ and } \hat{\beta} \text{ are LS estimates for the regression of } 
        Z \text{ on } X \\
          \text{ if and only if } \hat{\alpha} + v \hat{\beta} \text{ and } \hat{\beta} 
            \text{ are LS estimates} \\
              \text{ for the regression of } Z \text{ on } X - 1^{n} v. \\
                \text{Therefore, the set of LS planes for } (X - 1^{n} v, Z) \\
                  \text{ is exactly the same as that for } (X,Z).
    \end{multline}
(See \eqref{E:shft.invar.of.regrssn}.) In particular, $\Phi_{LS} (X , Z)$, when it exists, is not changed if we mean-center $X$. From \eqref{E:LS.criterion} we also see that 
    \begin{equation} \label{E:X.not.full.rank.then.beta.not.unique}
       \text{If } X \text{ is not of full rank } k \text{ then } \hat{\beta} \text{ is not unique.}
    \end{equation}

Let
    \begin{equation}  \label{E:X1.defn}
	X_{1} := (1^{n},  X)^{n \times (k+1)}    
    \end{equation}

Write $\alpha = (\alpha^{1}, \ldots, \alpha^{m})$. For $j = 1, \ldots, m$, let $\beta^{j}$ and $z^{j}$ be the $j^{th}$ columns of $\beta^{k \times m}$ and $Z^{n \times m}$, resp. Then \eqref{E:LS.criterion} can be written
	\begin{equation*}
		\sum_{i=1}^{n} | z_{i} - \alpha - x_{i} \beta |^{2} 
		  = \sum_{j=1}^{m} \left| z^{j} - X_{1}
			\begin{pmatrix}
				\alpha^{j} \\
				\beta^{j}
			\end{pmatrix} \right|^{2}
	\end{equation*}

Then $\hat{\alpha}$ and $\hat{\beta}$ are LS estimates for $Y$ if and only if the columns of
	\[
		X_{1} 
			\begin{pmatrix}
				\hat{\alpha} \\
				\hat{\beta}
			\end{pmatrix}
	\]
are the respective orthogonal projections of the columns of $Z$ onto the column space of $X_{1}$, which means columns of
		$Z - X_{1} 
			\begin{pmatrix}
				\hat{\alpha} \\
				\hat{\beta}
			\end{pmatrix}$
are perpendicular to the column space of $X_{1}$. In particular, LS estimates always exist 
and $\hat{\alpha}$ and $\hat{\beta}$ are LS estimates for $Y$ if and only if they satisfy the ``normal equations'' (Rice \cite[p.\ 476]{jaR88}): 
	\begin{equation}   \label{E:normal.eqns.Z}
	        X_{1}^{T}   X_{1} \,  (\hat{\alpha}^{T}, \, \hat{\beta}^{T} )^{T} =
	          X_{1}^{T}   X_{1}
			\begin{pmatrix}
				\hat{\alpha} \\
				\hat{\beta}
			\end{pmatrix}
	                   = X_{1}^{T}   Z.
	\end{equation} 
Hence, if $X_{1}$ is of full rank, e.g., if $X$ is mean-centered and $rank \, X = k$, then $X_{1}^{T}  X_{1}$ is invertible and 
	\begin{equation}  \label{E:LS.estimate.formula}
		( \hat{\alpha}^{T}, \, \hat{\beta}^{T} )^{T} = 
			\begin{pmatrix}
				\hat{\alpha} \\
				\hat{\beta}
			\end{pmatrix} =
		  ( X_{1}^{T}  X_{1} )^{- 1} X_{1}^{T} Z
	\end{equation}
(Anderson \cite[(10), p.\ 288]{Anderson}). Thus, if $X_{1}$ is of full rank the LS estimates 
$(\hat{\alpha}, \hat{\beta})$ are unique.

In this context, the expression ``$Y$ is collinear'' means the same thing as in definition \ref{D:collinearity}, viz.\ $x_{2} - x_{1}, \ldots, x_{n}- x_{1}$ do not span $\RR^{k}$. 
Thus, $Y$ is collinear if and only if $X - 1^{n} e_{1} X = (I_{n} - 1^{n} e_{1} )X$ has rank less than $k$. Here, $I_{n}^{n \times n}$ is the identity matrix and 
$e_{1}^{1 \times n} = (1, 0, \ldots, 0)$. More generally, we have the following generalization of \eqref{E:nothing.special.about.x1.in.collin}. Its proof can be found in appendix \ref{Chptr:misc.proofs}.

   \begin{lemma}  \label{L:collinearity.and.w.1xn}  
Let $X$ be a $n \times k$ matrix. If there exists $w^{n \times 1}$ s.t.\ 
$rank \, (X - 1^{n} w^{T} X) < k$ then $Y = (X, \, Z)$ is collinear. Conversely,  if $Y$ is collinear then for every $w^{1 \times n}$ with $w 1^{n} = 1$ we have $rank \, (X - 1^{n} w^{T} X) < k$. In particular, if $X$ is mean-centered and has rank $k$ then $Y$ is not collinear.
   \end{lemma}  
By the lemma, $Y = (X, Z)$ is collinear if and only if there exists $w^{n \times 1}$ s.t.\ $rank \, (X - 1^{n} w^{T} X) < k$, i.e., if and only if there is a plane $\zeta \subset \RR^{k}$ with $k' := \dim \zeta < k$ and the rows of $X$ lie on $\zeta + w^{T} X$. 
I.e., $Y$ is collinear if and only if the rows of $X$ lie on a plane in $\RR^{k}$ of dimension $ < k$. Note that in the first sentence of the preceding lemma $w$ need \emph{not} satisfy $w^{T} 1^{n} = 1$. 

  \begin{remark}[Linear regression is plane-fitting] \label{R:perf.fits.in.reg}
Let $Y = (X^{n \times k}, Z^{n \times m}) \in \Pf^{k}$ and suppose $Y$ is not collinear. Let $\bar{x}^{1 \times k} = n^{-1} 1_{n} X$ be the row vector of column means of $X$ 
and let $\bar{z}^{1 \times m} = n^{-1} 1_{n} Z \in \RR$ be the row vector of column means of $Z$. 
Thus, $X_{0} := X - 1^{n} \bar{x}$ and $Z_{0} := Z - 1^{n} \bar{z}$ are the mean-centered versions of $X$ and $Z$, resp. Since $Y \in \Pf^{k}$, by \eqref{E:when.is.Y.in.Pk}  
with $w = n^{-1} 1^{n}$, we have $rank \, (X_{0}, Z_{0}) = k$. On the other hand, since $Y$ is not collinear, by lemma \ref{L:collinearity.and.w.1xn}, we have $rank \, X_{0} = k$. Therefore, there exists a unique $B^{k \times m}$ s.t.\ $Z_{0} = X_{0} B$. Expanding that out we get 
$Z = (\bar{z} - \bar{x} B) 1^{n} + X B$. 

Consider a linear regression method $R$ which, given $Y' = (X', Z')$ computes, when possible, 
$a' \in \RR^{m}$ and $(B')^{k \times m}$ s.t.\ $a' 1^{n} + X' B'$ approximates $Z'$ as well as possible, in some sense. And suppose this is possible for $Y' \in \Y'$, where $\Y'$ is a dense subset of $\Y$ with $\Y' \cap \Pf$ dense in $\Pf$. Let $Y \in \Pf^{k}$ and suppose $Y$ is not collinear. From the preceding paragraph, we know that there exists a unique $B^{k \times m}$ s.t.\ $Z = (\bar{z} - \bar{x} B) 1^{n} + X B$. No better approximation to $Z$ is possible. Hence, $R(Y)$ is just the pair 
$(\bar{z} - \bar{x} B, B)$. Let $\Phi(Y) = \rho \bigl[ (I_{k}, B) \bigr]$. By \eqref{E:noncollinear.dense.inY} the set of collinear data sets has empty interior.
Therefore, by \eqref{E:plane-fitter.defn}, $\Phi$ is a plane-fitter providing few non-collinear data sets in $\Pf$ are singularities of $\Phi$. 

From Ellis \cite{spE98} we see that the preceding argument applies to show that the $\Phi$ corresponding to Least Median of Squares regression is a plane-fitter (remark \ref{R:robust.lin.reg}). That the same holds for Least Squares regression is a consequence of proposition \ref{P:collin.data.are.sings.of.LS}. It follows from \eqref{E:not.coll.in.Pk.not.LAD.sing} that the same is true of Least Absolute Deviation regression.
  \end{remark}

The proof of the following can be found in appendix \ref{Chptr:misc.proofs}.
   \begin{lemma}  \label{L:Y.collnr.iff.rank.X1.<.k+1}
$Y^{n \times \nvar} = (X^{n \times k}, Z^{n \times m})$ is collinear if and only if the rank of $X_{1} := (1^{n} , X)^{n \times (k+1)}$ is strictly less than $k+1$. This holds even if \eqref{E:n>nvar>k>0} is relaxed to allow $n = \nvar$. 
   \end{lemma}
Lemma \ref{L:Y.collnr.iff.rank.X1.<.k+1} means, by \eqref{E:LS.estimate.formula}, 
    \begin{equation}  \label{E:Y.not.collin.then.Phi(Y).unique}
      \text{If } Y \text{ is not collinear than } \Phi_{LS}(Y) \text{ exists uniquely.}
    \end{equation}

Recall from chapter \ref{Chptr:basic.setup} that singularity is always defined w.r.t.\ some dense subset of the data space $\D$. Recall that $X_{1} := (1^{n}  \;  X)^{n \times (k+1)}$. The following asserts, roughly speaking, that the closure of the image of any neighborhood of a singularity of LS contains the image of some Grassmannian. For proof see appendix \ref{Chptr:misc.proofs}. Recall \eqref{E:n.>.k+m}: $\nvar = k + m$.

   \begin{prop}   \label{P:collin.data.are.sings.of.LS}
Let $\Y'$ be the set of all non-collinear $\nvar$-dimensional data sets of the form 
$(X^{n \times k}, Z^{n \times m})$ to be used for regression of $Z$ on $X$. If $Y \in \Y'$ then the LS estimates for $Y$ are unique and $\Phi_{LS}$ is continuous 
on $\Y'$. $Y := (X, Z)$ is a singularity of LS (w.r.t.\ $\Y'$) if and only if $Y$ is collinear. In fact, if $Y$ is collinear and 
$rank \, X_{1} = k'+1 < k+1$, there is a $k'$-plane $\xi \in G(k', \nvar)$ and a  
linear imbedding $F : \RR^{k-k'+m} \to \RR^{k +m}$, both depending on $Y$, with the following properties. $F(\RR^{k-k'+m}) \cap \xi = \{ 0 \}$ 
(so $\RR^{\nvar} = \xi \oplus F(\RR^{k-k'+m})$) and for \emph{any} $\zeta \in G(k-k', k-k'+m)$ there is a family $\{ Y_{\epsilon, \zeta} \in \Y' : \epsilon > 0 \}$ of non-collinear data sets s.t.\ $Y_{\epsilon}$ converges to $Y$ and $\Phi_{LS}(Y_{\epsilon, \zeta})$ converges 
to $\xi \oplus F(\zeta)$ as $\epsilon \downarrow 0$. ($Y_{\epsilon}$ does not necessarily belong to $\Pf^{k}$.) 
   \end{prop}

In example \ref{Ex:all.LS.sings.are.90.degrees}, we will see that, at least if $m = 1$, all singularities of LS are severe in a precise sense. Clearly, $\D' = \Y'$ is invariant under rescaling. Therefore, by remark \ref{R:restricting.plane.fitter.to.sphere}, LS is suitable for localization (in a scale invariant fashion). 

Hence, by lemma \ref{L:dim.set.of.collin.data.sets}, when $m=1$ the dimension of the singular set of LS is $(n+1)k$. By lemma \ref{L:Y.collnr.iff.rank.X1.<.k+1}, if $Y$ is collinear then we must have $rank \, X_{1} < k+1$. See remark \ref{Ex:all.LS.sings.are.90.degrees} for further discussion of LS when $m = 1$.

Does LS have property \ref{Pty:agree.near.T} with $a$ equal to the dimension of its singular set? In any case, I conjecture that for LS the appropriate $R$ in theorem \ref{T:lwr.bnd.on.Haus.meas} is 0, at least if $m=1$. That is probably not hard to prove. 

 \begin{remark} \label{R:F.can.be.orthog.to.xi}
By \eqref{E:n.>.k+m}, $k+m = \nvar$ so $k-k'+m = \nvar - k'$. Let $F : \RR^{\nvar-k'} \to \RR^{\nvar}$ and $\xi \in G(k', \nvar)$ be as in the proposition. \emph{Claim:} We may assume that $F$ is an isometry s.t.\ 
$F(\RR^{\nvar-k'}) \perp \xi$. Recall \eqref{E:superscript.perp.notation}. Let $\Pi : \RR^{\nvar} \to \xi^{\perp}$ be orthogonal projection. Let $\omega \in G(k-k', \nvar)$ with $\omega \cap \xi = \{0\}$. E.g., $\omega$ might be $F(\zeta)$ for some 
$\zeta \in G(k-k', \nvar-k')$. Notice that $\Pi(\omega) \subset \xi^{\perp}$ 
is also a $(k-k')$-plane, obviously orthogonal to $\xi$. For suppose not. Then there exists $x \in \omega \setminus \{0\}$ s.t.\ $\Pi(x) = 0$. But $\Pi(x) = 0$ means $x \in \xi$, contradicting $\omega \cap \xi = \{0\}$.  Moreover, $\xi + \omega = \xi + \Pi(\omega)$. 

Let $F' := \Pi \circ F : \RR^{\nvar-k'} \to \xi^{\perp}$. $F'$ has full rank $\nvar-k'$. For assume not. Then, we can construct $\pi \in G(k-k', \nvar-k')$ s.t.\ $F'$ maps a nonzero vector in $\pi$ to 0. Let $\omega := F(\pi)$. Then by assumed properties of $F$, 
$\omega \in G(k-k', \nvar)$ and $\omega \cap \xi = \{0\}$. Thus, there exists $x \in \omega \setminus \{0\}$ s.t.\ $\Pi(x) = 0$. But in the last paragraph we showed this is impossible. Conclusion: $F'$ has full rank $\nvar-k'$. Thus, we may use $F'$ in place of $F$. 

Since $\dim \xi^{\perp} = \nvar-k'$, there exists an isometry 
$F'' : \RR^{\nvar-k'} \to \xi^{\perp}$. Suppose $\zeta \in G(k-k', \nvar-k')$. 
Let $\pi := (F'')^{-1} \bigl[ F'(\zeta) \bigr]$. Then $\pi \in G(k-k', \nvar-k')$ and 
$F''(\pi) = F'(\zeta)$. In the other direction. Suppose $\pi \in G(k-k', \nvar-k')$ and let $\zeta := (F')^{-1} \circ F''(\pi)$. Then $F'(\zeta) = F''(\pi)$. Thus, we may use $F''$ in place of $F'$ in place of $F$. That concludes the proof of the claim that we may assume that $F$ is an isometry s.t.\ $F(\RR^{\nvar-k'}) \perp \xi$. 
  \end{remark}

The following makes \eqref{E:noncollinear.dense.inY} more precise.
    \begin{lemma} \label{L:dim.set.of.collin.data.sets}
      If $m = 1$, the dimension of the set of collinear data sets 
        $= \dim \Pf -1 = (n+1)k \leq n \nvar - 2$.
    \end{lemma}

(This appears as \cite[Example 2.8]{spE95}. See \eqref{L:Pf.is.a.manif}.) It follows from proposition \ref{P:collin.data.are.sings.of.LS}, that the dimension of the singular set of Least Squares (with $m=1$) is $(n+1)k$. Later (subsection \ref{SS:gnrl.lwr.bnd.plane.fit}, we will see that, when 
$\nvar = k+1$, $(n+1)k$ is the smallest dimension that the singular set of a plane-fitter can have.
   \begin{proof}[Proof of lemma \ref{L:dim.set.of.collin.data.sets}]
As usual given $X^{n \times k}$ we denote its rows by $x_{1}, x_{2}, \ldots, x_{n}$. Consider the map $F$ which takes $X$ to the matrix of the same size with rows $x_{1}, x_{2} - x_{1}, \ldots, x_{n} - x_{1}$. $F$ is invertible and, by example \ref{Ex:ratnl.fns.loc.Lip}, bi-Lipschitz. (See \eqref{E:bi-Lipschitz.defn}.) Let $\mcl{C}$ denote the set of collinear data sets. By \eqref{E:Lip.magnification.of.Hm}, the dimension of $\mcl{C}$ is therefore the same as that of 
$\mcl{C}' := \bigl\{ (F(X), y) : (X,y) \in \mcl{C} \bigr\}$. Let $G$ take $X^{n \times k}$ to the $(n-1) \times k$ consisting of the last $n-1$ rows of $F(X)$, vix.\ . $x_{2} - x_{1}, \ldots, x_{n} - x_{1}$. If $Y = (X,y)$ is collinear, then $rank \, G(X) \leq k-1$. Hence, by lemma \ref{L:dim.of.low.rank.mats} with $\mu = n-1$, $\nu=k$, and $r = k-1$; the dimension 
of $\bigl\{ G(X) :(X,y) \in \mcl{C} \bigr\}$ is no bigger than $(k-1)(n-1) + (k-1)k - (k-1)^{2} = n(k-1)$. Collinearity of $Y$ puts no constraints on $x_{1}$. Thus, $x_{1}$ can be an arbitrary element of $\RR^{k}$. Therefore, by lemma \ref{L:dim.of.product}. the dimension of $F(X)$  is $nk - n + k$. Similarly, collinearity of $Y$ puts no constraints on $y^{n \times 1}$ so 
$\dim \mcl{C} = \dim \mcl{C}' = (nk - n + k) + n = (n+1)k$. By lemma \ref{L:Pf.is.a.manif} 
with $\nvar = k+1$, this is 1 less than $\dim \Pf^{k}$. 
   \end{proof}

  \begin{remark}[Dimension of singular set for least squares]  \label{R:dim.S.for LS}
At least if $n > \nvar + 1$, any $\T$ constructed as in section \ref{SS:D.T.plane.fit}, must contain a collinear data set. For otherwise, in the setting of $\D_{\infty}$ in section \ref{SS:D.T.plane.fit}. LS would satisfy the hypotheses of proposition \ref{P:sing.codim.in.plane-fitting} and its singular set would have codimension no greater than 2. But the inequality in lemma \ref{L:dim.set.of.collin.data.sets} is strict 
if $n > \nvar + 1$.

However, proposition \ref{P:collin.data.are.sings.of.LS} also tells us that proposition \ref{P:sing.codim.in.plane-fitting} does not apply to LS because the $\T$ in section \ref{SS:D.T.plane.fit}, always contains a collinear data set, i.e., it always contains a singularity of LS. (See panel (LS,a) in figure \ref{F:LS.PC.LAD.lf.plots} for an example.) But, by \eqref{E:r=1} and \eqref{E:in.plane.fitting.T.is.smooth.submanif}, in plane-fitting, $r = t = 1$ so \textbf{hypothesis \ref{Hyp:S.cap.T.small}} of theorem \ref{T:Phi.star.Hr.contains.Theta.star.Hr} therefore requires 
$\Ss \cap \T = \varnothing$. Therefore, \textbf{hypothesis \ref{Hyp:S.cap.T.small}} of theorem \ref{T:Phi.star.Hr.contains.Theta.star.Hr} does not hold for LS and so proposition \ref{P:sing.dim.when.H.d-r.D.=.0} does not either. We already saw this in figure \ref{F:sing.dists.cdfs}, where the codimension of the singular set of LS was 3. In section \ref{SS:gnrl.lwr.bnd.plane.fit} we compute a bound on singular set dimension for plane-fitting that does not require \textbf{hypothesis \ref{Hyp:S.cap.T.small}}. The bound computed in that section is tight for LS when $m = \nvar - k = 1$. 
  \end{remark}

  \begin{remark}[LS in ``long`` vs.\ ``wide`` data]   \label{R:LS.in.long.vs.wide}
In Statistics there is considerable interest in analysis of ``wide'' data sets, i.e., data sets in which $k$ is large relative to $n$ (Hall \emph{et al} \cite{pHjsMaN05.HiDimLowN}). Lemma \ref{L:dim.set.of.collin.data.sets} can shed some light on this issue. Suppose $m=1$ 
so $k = \nvar -1$. Consider pairs $(n, \nvar)$ with $\text{total amount of data} = n \times \nvar = C$, a constant. 
I.e $\dim \Y = C$. 
 
By proposition \ref{P:collin.data.are.sings.of.LS}, the singular set, $\Ss$, of LS is the set of collinear data sets. By lemma \ref{L:dim.set.of.collin.data.sets},
    \begin{equation*}
      \dim \Ss = n k + k = n(\nvar - 1) + k = C - C/\nvar + k = C - C/(k+1) + k. 
    \end{equation*}
Since $C$ is constant, this expression is increasing in $k$. This suggests that LS is less stable on ``wide'' data sets ($k$ large relative to $n$) than on ``long'' ones ($k$ small relative to $n$). (See section \ref{SS:asymp.prob}.)  See remark \ref{R:long.vs.wide} for further discussion of long vs.\ wide..
  \end{remark}

\section{Miscellaneous remarks} \label{SS:remarks}
  \begin{remark}[``Long'' vs.\ ``wide'' data]  \label{R:long.vs.wide} As stated in remark \ref{R:LS.in.long.vs.wide}, an important issue in Statistics is the effect of the relative sizes of $n$ and $\nvar$ on the performance of statistical methods. The issue is related to the ``curse of dimensionality'', (Bellman \cite[Section  5.16]{reB61}, Hastie \emph{et al} \cite[section 2.5]{tHrTjF01.statlearn}). (Incidentally, Bellman \cite{reB57} is sometimes cited as a reference on the ``curse of dimensionality''. However, in my cursory examination of that work I found that Bellman \cite{reB57} only seems to mention it in passing in the preface, on p.\ ix.)  If $n$ is much larger than $k$ then $x$ is ``long''. If $n$ is not much bigger than $\nvar$ or even smaller than $\nvar$ then the data set is ``wide''. Most statistical theory deals with the ``long'' case. Comparison of singular sets of LS in long and wide data is discussed in remark \ref{R:LS.in.long.vs.wide} above. Here we discuss it for general linear regression.

It seems that the theory developed in this book might help compare data analysis in the long vs.\ wide cases. Recall that $\Pf^{k}$ is the collection of all data sets whose rows lie exactly 
on a unique $k$-plane (not necessarily through the origin). 

 Let $m = 1, 2, \ldots,$ be constant. Let $0 < k_{2} < k_{1}$. As in \eqref{E:n.>.k+m} let $\nvar_{i} = k_{i} + m$. Suppose $n_{1} \nvar_{1} = d = n_{2} \nvar_{2}$. Thus, $\nvar_{1} > \nvar_{2}$ so $n_{2} > n_{1}$. Then working in $\D_{\infty}$ as in section \ref{SS:D.T.plane.fit}, the data spaces, $\D_{i}$, for $i=1,2$ are  the same, 
but if $Y \in \D_{1}$ then $Y$ is relatively ``wide'' while if $Y \in \D_{2}$ then $Y$ is ``long''. Let $\Pf^{k_{i}}$ be the corresponding space of perfect fits. Then, by lemma \ref{L:Pf.is.a.manif},  
	\begin{equation*}
		\dim \Pf^{k_{i}} = n(\nvar_{i} - m) + (k_{i} + 1)m = d - n_{i}m + mk_{i} + m 
	          \qquad (i=1,2).
	\end{equation*}
Therefore,
    \begin{multline*}
       \dim \Pf^{k_{1}} - \dim \Pf^{k_{2}} 
         \approx \bigl[ d - n_{1}m + mk_{1} + m \bigr] - \bigl[ d - n_{2}m + mk_{2} + m \bigr] \\
           = m \bigl[ (n_{2} - n_{1}) + (k_{1} - k_{2}) \bigr] > 0.
    \end{multline*}  

Suppose for $R > 0$, $\Phi_{i,R}$ is a plane fitting method on $\D_{i}$ with $k_{i}$ predictors and having singular set $\Ss_{i,R}$ ($i=1,2$) s.t.\ $(\Phi_{i,R}, \Ss_{i,R}, etc.)$ has property \ref{Pty:agree.near.T} with $a = d-2$ (see \eqref{E:big.lwr.bound.in.plane.fit}) and 
$dist_{d-2}(\Ss_{i,R}, \Pf^{k_{i}}) = R$. (See \eqref{E:essential.dist.defn}.) Let $\Pf^{k_{i}}$ play the role of $\Pf$ in theorem \ref{T:lwr.bnd.on.Haus.meas}. (See example \ref{Ex:basic.property.case}.)
Suppose the inequality \eqref{E:Hm.a.S.geq.R.d-p-1} is an equality 
when applied to $(\Phi_{i,R}, \Ss_{i,R}, \Pf^{k_{i}})$, with possibly different values 
of the constant $\gamma$. Then the exponent in \eqref{E:Hm.a.S.geq.R.d-p-1} 
is smaller for $\Phi_{1,R}$ than it is for $\Phi_{2,R}$. Now, the constant $\gamma$ may depend on $i$ ($i=1,2$), but still for $R$ sufficiently small we have 
	\[
		\Hm^{d-2}(\Ss_{2,R}) < \Hm^{d-2}(\Ss_{1,R}) .
	\]
This suggests that $\Phi_{2,R}$ is more stable on ``long'' data sets than on ``wide''. This is just a ``hand waving'' argument. A similar but complete rigorous argument of this sort is carried out in chapter \ref{Chptr:robst.loc.on.circle} for a different statistical problem. (See proposition \ref{P:aug.direct.mean.beats.robst.loc}.)

(It seems like a similar argument goes through with $n_{1} = n_{2}$ but $k_{2} < k_{1}$. However, in that case the dimension of the data space, $\D_{2}$, is smaller, so it is not surprising 
that $\Ss_{2,R}$ is also smaller. In the case we considered, the dimensions of the two data spaces are the same.)
  \end{remark}

  \begin{remark}[Transformed variables]  \label{R:transformed.vars}
In linear regression it is common to add terms to the regression model that are nonlinear functions of the variables. So then the regression model takes the form
	\[
		y = b_{0} + \sum_{i=1}^{k} b_{i} x_{i} + \sum_{i=1}^{m} b_{k+i} f_{i}(\mbf{x}),
	\]
where $\mbf{x} = (x_{1}, \ldots, x_{k})$ is the vector of predictor values and 
$f_{i} : \RR^{k} \to \RR$. A common choice for 
the functions $f_{i}$ are polynomials (Draper and Smith \cite[chapter 5]{nDhS81.draper.and.smith}), e.g., $f_{\ell}(\mbf{x}) = x_{i} x_{j}$. This extension of linear regression is easily handled by the theory in this section. The space, $\Pf^{k}$ can still be used. Data sets in $\Pf^{k}$ correspond to $b_{k+1} = \cdots b_{\nvar} = 0$. 
The projection  $(b_{1}, \ldots, b_{k}, b_{k+1}, \ldots, b_{\nvar}) \mapsto (b_{1}, \ldots, b_{k})$ thus defines a plane-fitter and singularity of $(b_{1}, \ldots, b_{k})$ is a form of singularity of $(b_{1}, \ldots, b_{k}, b_{k+1}, \ldots, b_{\nvar})$.

It is also important to note that adding the transformed variables to the model does not increase $\dim \D$. These issues come up again in remark \ref{R:nonparam.reg}.
  \end{remark}

  \begin{remark}[Plane-fitting in vector bundles]   \label{R:plane.fitting.on.vec.bndls}
We can generalize plane-fitting in the following fashion. 
Let $\zeta$ be a $\nvar$-plane bundle over a paracompact base space, $\D$ (Milnor and Stasheff \cite[\S\S 2, 5.8]{jwMjdS74}). 
By Milnor and Stasheff \cite[Theorem 5.6, p.\ 65]{jwMjdS74} there is a bundle map (Milnor and Stasheff \cite[p.\ 26]{jwMjdS74})
$H : E(\zeta) \to E(\gamma^{\nvar})$, where $\gamma^{\nvar}$ is the 
``universal $\nvar$-plane bundle'' (Milnor and Stasheff \cite[p.\ 63]{jwMjdS74}) and ``$E$'' indicates total space. 
Let $\Gamma(\zeta)$ be the fiber bundle 
(Spanier \cite[pp.\ 90--91]{ehS66}, Husemoller  \cite[pp.\ 11--15]{dH75.FibreBundles}) whose fiber over $x \in \D$ is the Grassmann manifold consisting of all
$k$-planes through the origin in $\pi^{-1}(x) \homeomto \RR^{\nvar}$, 
where $\pi : E(\zeta) \to \D$ is the bundle projection map 
of $\zeta$. Denote the total space of $\Gamma(\zeta)$ by 
$E \bigl[ \Gamma(\zeta) \bigr]$.
Define $\Gamma(\gamma^{\nvar})$ and $E \bigl[ \Gamma(\gamma^{\nvar}) \bigr]$ similarly. 
$H$ induces a bundle morphism $\hat{H} : E \bigl[ \Gamma(\zeta) \bigr] \to E \bigl[ \Gamma(\gamma^{\nvar}) \bigr]$.
Let $\pi: E \bigl[ \Gamma(\zeta) \bigr] \to \D$ also denote the obvious 
projection in $\Gamma(\zeta)$. Suppose $\D' \subset \D$ and 
$\Phi : \D' \to E \bigl[ \Gamma(\zeta) \bigr]$ is continuous and satisfies $\pi \circ \Phi(x) = x$, 
$x \in \D'$. (I.e., $\Phi$ is a section of the restriction of $\Gamma$ to $\D'$.)

We can think of a point of $E \bigl[ \Gamma(\gamma^{\nvar}) \bigr]$ as a pair, $(X,Y)$, 
where $X$ is a point in the infinite Grassmannian $G_{\nvar}$ (Milnor and Stasheff \cite[p.\ 63]{jwMjdS74}) and $Y$ is a
$k$-dimensional subspace of $X$. But $X$ is a $\nvar$-dimensional subspace 
of $\RR^{\infty}$. 
Thus, $Y$ is a $k$-dimensional subspace of $\RR^{\infty}$. Hence, there is a projection
$g: E \bigl[ \Gamma(\gamma^{\nvar}) \bigr] \to G_{k}$. Consider the plane-fitter. 
$\Phi_{\infty} = g \circ \hat{H} \circ \Phi : \D' \to G_{k}$.
Then we can try to apply our theory with $\Phi = \Phi_{\infty}$ and $\F = G_{k}$. The tricky part is checking condition 
\eqref{E:nontriv.r-dim.homol}. However, if $\D$ is compact then by Milnor and Stasheff \cite[Lemma 5.3, p.\ 61]{jwMjdS74} we can replace 
$G_{k}$ by a finite dimensional Grassmannian. Then the theory of this chapter applies directly. 
  \end{remark}

\subsection{Function-valued maps}  \label{SS:functions.vs.geometry}
This chapter applies in particular to linear regression or a linear systems solving method viewed as a plane-valued operation. In practice, however, a linear regression method is usually viewed as a vector- or function-valued operation in which the components of the vector are the coefficients in the real scalar-valued affine function. Let $R$ be such a method and let $\Phi$ be the associated plane-fitter. Here we show that any data set that is a singularity of $\Phi$ is also a singularity of $R$. Let $\D'$ be a dense subset of $\D$ s.t.\ $\D' \cap \Pf^{k}$ is dense in $\Pf^{k}$. 

If $a \in \RR$ and $b^{k \times 1} \in \RR^{k}$ define 
$f : (a,b^{T}) : (a,b^{T}) \mapsto (x \mapsto a + x b , \; x \in \RR^{k}$), where 
$\nvar = k+1$. So $f$ maps $\RR^{\nvar}$ to an affine function. (One might topologize the space of such functions by uniform convergence on unit ball in $\RR^{k}$. Note that that space is not complete, remark \ref{R:.completeness.of.F}.)
Let $\phi : Y \to \bigl( a(Y)^{1 \times 1}, [b(Y)^{T}]^{1 \times k} \bigr) \in \RR^{\nvar}$ 
($Y \in \D'$). Suppose $\phi$ is continuous on $\D'$. By \eqref{E:in.lin.reg.Phi.is.graph},
$\Phi(Y) \in G(k, \nvar)$ is the row space of $(I_{k}^{k \times k}, b(Y))$, where $I_{k}$ is 
the $k \times k$ identity matrix, is the $k$-plane through the origin parallel to the graph of $f \bigl[ (a(Y),b(Y)^{T}) \bigr]$. By \eqref{E:convergence.in.Grassmann}, the composition $g : (a,b^{T}) \mapsto b \mapsto (I_{k}, b) \mapsto \rho(I_{k}, b)$ is continuous. 
Thus, $\Phi = g \circ \phi$. Therefore, if $Y$ is a singularity 
of $\Phi$ (w.r.t.\ $\D'$), then \emph{a fortiori} it is a singularity of $\phi$. 

This phenomenon generalizes. Let $\F$ and $\F'$ be topological spaces and suppose we are interested in a map $\phi : \D \partlyto \F'$. Suppose there is a continuous map 
$g : \F' \to \F$. Consider the composition $\Phi := g \circ \phi : \D \partlyto \F$. Let $\Ss$ be the singular set of $\Phi$. Then $\Ss$ is a subset of the singular set of $\phi$.

The converse is false. I.e., $Y \in \D$ might be a singularity of $\phi$ but not of $\Phi$. We illustrate this in the regression setting. First, notice that 
$\RR^{\nvar} = \bigl\{  (a^{1 \times 1} , (b^{T})^{1 \times k}) \bigr\}$ with the Euclidean norm is homeomorphic to the space of affine functions on $\RR^{k}$ with sup norm on the unit ball 
in $\RR^{k}$ or the $L^{1}(\mu)$ norm, where $\mu$ is a measure with finite first absolute moment. So if suffices to consider the Euclidean norm on $\RR^{\nvar}$. Note that at a singularity not only is the representation of the function, i.e., the coefficient vector, unstable, but the function itself is unstable because $(a,b)$ can be recovered from $f(a,b)$ by a continuous operation.

Suppose $\phi : \D' \to \RR^{\nvar}$ maps each $Y \in \D'$ to a pair $(a,b^{T}) \in \RR^{\nvar}$, where $a \in \RR$ and $b^{k \times 1} \in \RR^{k}$. Suppose $\phi$ satisfies the analogue of \eqref{E:plane.fitting.constraint}: First suppose that for \emph{every} $Y \in \Pf$ the vector 
$\phi(Y)$ is defined and the graph of $f \bigl[ \phi(Y) \bigr]$ is parallel to $\Delta(Y)$. 
Now let $Y \in \Y$ be collinear (definition \ref{D:collinearity}). \emph{Claim:} $Y$ is a singularity 
of $\phi$. Write $Y = (X^{n \times k}, y^{n \times 1})$ as in subsection \ref{SS:lin.reg.and.LS}. 
Let $x_{i}^{1 \times k}$ be the $i^{th}$ row 
of $X$ and $y_{i} \in \RR$ be the $i^{th}$ entry in $y$ ($i=1, \ldots, n$). Since $Y$ is collinear, by definition \ref{D:collinearity}, the matrix $Z^{(n-1) \times k}$ whose $i^{th}$ row is 
$x_{i+1} - x_{1}$ ($i=2, \ldots, n$) has rank $< k$. Let $z^{(n-1) \times 1}$ be the column vector 
$(y_{2}-y_{1}, \ldots, y_{n}-y_{1})^{T}$. 

Let $\epsilon > 0$. Let $c^{k \times 1}$ be a unit vector s.t.\ $Z c = 0$ and let 
$W_{1}^{(n-1) \times k} =  \bigl( z+ \sqrt{\epsilon} 1^{n-1} \bigr) c^{T}$, where, as usual, 
$(1^{n-1})^{(n-1) \times 1}$ is the $(n-1)$-dimensional column vector of 1's. 
Thus, $W_{1} c = z+ \sqrt{\epsilon} 1^{n-1}$. Note that there exists $\epsilon_{0} > 0$ 
s.t.\ for $\epsilon \in (0, \epsilon_{0})$ we have $z+ \sqrt{\epsilon} 1^{n-1} \neq 0$ (even if $z=0$). Thus, for $\epsilon \in (0, \epsilon_{0})$ we have $W_{1} c \neq 0$ and 
$\rho(W_{1}) \subset \RR^{k}$ is just the one-dimensional space spanned by $c^{T}$. 

We have $rank \, Z < k$. If $rank \, Z = k-1$ then let $W_{2}^{(n-1) \times k} := 0$. Otherwise, let $V \subset \RR^{k}$ be the orthogonal complement of $c$, so $\dim V = k-1$. 
Choose $W_{2}^{(n-1) \times k}$ to satisfy $W_{2} \, c = 0$ 
(so $\rho(W_{2}) \subset V$ ) and also satisfy $V = \rho(W_{2}) \oplus \rho(Z)$. 
Then, if $\epsilon \in (0, \epsilon_{0})$, 
    \begin{equation} \label{E:eps.W1+eps.W2+Z}
      \epsilon \, W_{1} + \epsilon \, W_{2} + Z \text{ has rank } k .
    \end{equation}

Let $W = W_{1}+W_{2}$ and, for $\epsilon > 0$, let 
	\begin{equation*}
		Y_{c, \epsilon}^{n \times \nvar} := (X_{c, \epsilon}, y_{c, \epsilon}) :=
			\left( 
			                1^{n} x_{1} + \epsilon
					\begin{pmatrix}
						0^{1 \times k} \\
						W
					\end{pmatrix} 
				+
					\begin{pmatrix}
						0^{1 \times k} \\
						Z
					\end{pmatrix} \, , \,
				y + \sqrt{\epsilon}
					\begin{pmatrix}
						0^{1 \times 1} \\
						1^{n-1}
					\end{pmatrix}
			\right).
	\end{equation*}
Then $Y_{c, \epsilon} \to Y$ as $\epsilon \to 0$. The first row of $X_{c, \epsilon}$ is $x_{1}$. By \eqref{E:eps.W1+eps.W2+Z}, subtracting $x_{1}$ from rows 2 through $n$ yields a matrix of rank $k$. Therefore, $rank \, X_{c, \epsilon} = k$.  

Let $(w^{n \times 1})^{T} := (1, 0, \ldots, 0)$. Then
    \begin{equation*}
      Y_{c, \epsilon} - 1^{n} w^{T} Y_{c, \epsilon} = 
      			  \left(
        					\begin{pmatrix}
    						0^{1 \times k} \\
    						\epsilon W
    					\end{pmatrix} 
    				+
    					\begin{pmatrix}
    						0^{1 \times k} \\
    						Z
    					\end{pmatrix},
    					\begin{pmatrix}
    						0^{1 \times 1} \\
    						z + \sqrt{\epsilon} 1^{n-1}
    					\end{pmatrix}
    			\right).
    \end{equation*}
By choice of $c$ and definitions of $W_{1}$, $W_{2}$, and $W$, we have
 $\epsilon^{-1} (\epsilon W + Z) c = W_{1} c = z + \sqrt{\epsilon} 1^{n-1}$. 
Thus, the last column of $Y_{c, \epsilon} - 1^{n} w^{T} Y_{c, \epsilon}$ is in the span of the first $k$ columns. Therefore, by \eqref{E:eps.W1+eps.W2+Z}, 
$Y_{c, \epsilon} - 1^{n} w^{T} Y_{c, \epsilon}$ has rank $k$. Hence, by \eqref{E:when.is.Y.in.Pk}, $Y_{c, \epsilon} \in \Pf^{k}$. But we are currently assuming that for every $Y \in \Pf$ the vector 
$\phi(Y)$ is defined and the graph of $f \bigl[ \phi(Y) \bigr]$ is parallel to $\Delta(Y)$. 
So $\phi(Y_{c, \epsilon})$ is defined.

Let $a := a_{c, \epsilon} := - \epsilon^{-1} x_{1} c + y_{1} \in \RR$ and $b^{k \times 1} := b_{c, \epsilon} := \epsilon^{-1} c$. Then it is easy to see that 
    \begin{equation} \label{E:a.1n.+.X.b.=.y}
        a 1^{n} + X_{c, \epsilon} \, b = y_{c, \epsilon}.
    \end{equation} 
Therefore $\phi(Y_{c, \epsilon}) = (something, b_{c, \epsilon})$. (The 
``$something \in \RR$'' may or may not be $a_{c, \epsilon}$.) By assumption, the graph of $f \bigl[ \phi(Y_{c, \epsilon}) \bigr]$ is parallel to $\Delta(Y_{c, \epsilon})$. 
I.e, $\rho(I_{k}, b_{c, \epsilon}) = \Delta(Y_{c, \epsilon})$. (See \eqref{E:in.lin.reg.Phi.is.graph}.)
 
Now, as $\epsilon \to 0$ the coefficient vector $b_{c, \epsilon}$ shoots off to infinity. But we may replace $c$ by $-c$, in which case as $\epsilon \to 0$ the coefficient vector $b_{c, \epsilon}$ shoots off in the opposite direction. This completes the proof of the claim that $Y$ is a singularity of $\phi$, under the assumption that for \emph{every} $Y \in \Pf$ the vector $\phi(Y)$ is defined and the graph 
of $f \bigl[ \phi(Y) \bigr]$ is parallel to $\Delta(Y)$. 

Suppose that only for $Y$ in a dense subset, $\Pf'$, of $\Pf$ is it the case that 
$\phi(Y)$ is defined and the graph of $f \bigl[ \phi(Y) \bigr]$ is parallel 
to $\Delta(Y)$. Then replace $Y_{c, \epsilon}$ by 
$Y_{c, \epsilon} + o(\epsilon) \in \Pf'$. Then the preceding argument will still go through and we find that every collinear data set is a singularity of $\phi$.

Contrast this with proposition \ref{P:few.collin.LAD.sings}, according to which \emph{most collinear data sets are} not \emph{singularities of LAD} (example \ref{Ex:3.plane.fitters} and section \ref{SS:LAD}), but as we have just seen every collinear data set is a singularity of LAD regarded as a vector- or affine function-valued map. Thus, the singular set of $\phi$ can be a proper superset of that of the corresponding plane-fitter. So as data approaches most collinear data sets, the graph of the LAD regression function converges to a plane (space of planes, $G(\nvar-1, \nvar)$, is complete), but the regression function itself tries to converge to a non-function, one with a ``vertical'' graph (space of functions, with $\sup$ norm, say, is not complete). This possibility was raised in remark \ref{R:.completeness.of.F}. 

Subsection \ref{SSS:reg.coefs,near.90.degree.sings} also concerns instability of regression coefficients. 

\section{Least absolute deviation linear regression} \label{SS:LAD}
In this section we treat stability of LAD regression in some depth. (Dodge and Roenko \cite{yDnR92.L1.stability} examine a different form of stability of LAD.) 
Let $Y = (X,y)^{n \times \nvar} \in \Y$ as in section \ref{SS:lin.reg.and.LS}, specifically \eqref{E:initially.in.lin.reg.nvar=k+1}. (Thus, in this section 
$\nvar = k + 1$.) As usual, denote the $i^{th}$ row of $X$ by $x_{i}^{1 \times k}$ and the $i^{th}$ element of $y$ by $y_{i}$. Recall that in $L^{1}$ or Least Absolute Deviation (LAD, Bloomfield and Steiger \cite{pBwlS83}) regression one fits to $Y$ a plane 
$y = \beta_{0}^{1 \times 1} + x^{1 \times k} \beta_{1}^{k \times 1}$  
($x \in \RR^{k}$), where 
$b^{(k+1) \times 1} = (b_{0}, b_{1}^{T})^{T} = (\beta_{0}, \beta_{1}^{T})^{T}$ minimizes 
	\begin{equation}   \label{E:L1.criterion.defn}
		L^{1}(b,Y) := \sum_{i=1}^{n} | y_{i} - b_{0} - x_{i} b_{i} |.
	\end{equation}
($b_{0}, \beta_{0} \in \RR$ and $b_{1}, \beta_{1}$ are $k \times 1$.) In this case say that the $k$-plane $\{ (x, \beta_{0} + x \beta_{1}): x \in  \RR^{k} \}$ is an ``LAD plane'' and $\beta$ an ``LAD solution'' for $Y$. Let  $\hat{B}(Y)$ denote the set of all $\beta$'s minimizing $L^{1}(\beta, Y)$. Write $\beta(Y) = \beta^{1 \times (k+1)}$ whenever there is only one vector $b = \beta^{T}$ minimizing $L^{1}(b, Y)$. By lemma \ref{L:basic.LAD.soln.facts}(a) in appendix \ref{Chptr:LAD.technicalities},  
$\hat{B}(Y)$ is nonempty, compact, and convex. If $\hat{B}(Y)$ is a singleton, denote the element of $G(k, \nvar)$ parallel to the unique LAD plane 
by $\Phi(Y) = \Phi_{LAD}(Y)$. 
  
Recall that $Y \in \Y$ is ``(multi)collinear'' (definition \ref{D:collinearity}) 
if $x_{2} - x_{1}, \ldots, x_{n}- x_{1}$ do not span $\RR^{k}$ 
(by \eqref{E:n>nvar>k>0}, $n > \nvar$.) 
   \begin{multline} \label{E:Y.LAD'.defn}
	\text{Let  $\Y_{LAD}'$ denote the set of all $Y \in \Y$ s.t.\ $Y$ is not collinear} \\
	    \text{ and $\hat{B}(Y)$ contains exactly one point.}  
   \end{multline}
So $\Phi_{LAD}$ is defined everywhere in $\Y_{LAD}'$. Then by proposition \ref{P:Y'.LAD.is.dense} in appendix \ref{Chptr:LAD.technicalities} we have that 
$\Y_{LAD}'$ is dense in $\Y$. (As observed just after proposition \ref{P:Y'.LAD.is.dense}, $\Y_{LAD}' \cap \D$ is dense in $\D$ defined by \eqref{E:plane.fitting.D.mu.defn}.) Clearly, 
$\D' = \Y_{LAD}'$ is invariant under rescaling. Therefore, by remark \ref{R:restricting.plane.fitter.to.sphere}, LAD is suitable for localization (in a scale invariant fashion).

\emph{Claim:} $\Phi$ is continuous on $\Y_{LAD}'$. Let $Y \in \Y_{LAD}'$. By lemma \ref{L:basic.LAD.soln.facts}(a) there is a neighborhood $\mcl{V} \subset \D$ of $Y$ and a compact set $C \subset \RR^{k+1}$ s.t.\ if $Y' \in \mcl{V}$ then $Y'$ is not collinear and $\hat{B}(Y') \subset C$. Now, $\Y_{LAD}' \cap \mcl{V} \neq \varnothing$ and is precisely the set of data sets $Y'$ in $\mcl{V}$ s.t.\ the LAD optimization problem has a unique solution, $\beta(Y')$, and that solution is always in the compact set $C$. By lemma \ref{L:data.maps.defined.by.opt}, part \eqref{I:cmpct.opt} with $\D = \mcl{V}$ and $\F = C$ it follows that $\beta(Y)$ is continuous 
in $Y \in \Y_{LAD}'$ as claimed. A consequence is:
    \begin{multline} \label{E:unique.LAD.soln.means.not.sing}
      \text{If $Y \in \Y$ is not collinear and has a unique LAD solution } \\
        \text{(i.e., } Y \in \Y_{LAD}') \text{ then $Y$ is not a singularity of LAD 
          w.r.t.\ } \Y_{LAD}' .
    \end{multline}

Therefore, \eqref{E:D'.dense.Phi.cont.on.D'} holds. Hence, by lemma \ref{L:extend.Phi.to.D.less.S}, we may replace  $\Y_{LAD}'$ by $\D' \supset \Y_{LAD}'$ satisfying \eqref{E:D'.=.D.less.S} with $\Ss$ the singular set of $\Phi$ w.r.t.\ $\Y_{LAD}'$. 
(See also corollary \ref{C:nonunique.mins.are.sings} in appendix \ref{Chptr:LAD.technicalities}.) By proposition \ref{P:few.collin.LAD.sings}, $\D'$ will contain most collinear data sets. Since $\Y_{LAD}' \subset \D'$ any singularity of LAD w.r.t.\ $\Y_{LAD}'$ is a singularity w.r.t.\ $\D'$.

By \eqref{E:Perfect.Fits.in.plane.fitting}, $\Pf^{k}$ is the collection of all data sets (i.e., $n \times \nvar$ matrices) whose rows lie exactly on a unique $k$-plane (not necessarily through the origin). \emph{Claim:} 
    \begin{equation}  \label{E:noncollin.in.Pf.is.in.Y.LAD'}
      \text{Any non-collinear data set in $\Pf^{k}$ is in } \Y_{LAD}' . 
    \end{equation}
For suppose $Y = (X,y)$ is such a data set. Since $Y$ is noncollinear, by lemma \ref{L:Y.collnr.iff.rank.X1.<.k+1}, 
$(1^{n} , X)^{n \times (k+1)}$ has full rank $k+1$. That means 
    \begin{equation*}
      rank \, X = k .
    \end{equation*}  

Since $rank \, (1^{n} , X) = k+1$ but $rank \, X = k$, we have that $1^{n}$ does not lie in the column space, $\gamma(X)$, of $X$. So if $v$ is the orthogonal projection of $1^{n}$ onto $\gamma(X)$ then $v \neq 1^{n}$. Let $w' := 1^{n} - v \neq 0$. Then $w'$ belongs to the column space, $\gamma \bigl[ (1^{n} , X) \bigr]$, of $(1^{n} , X)$. 
Since $rank \, (1^{n} , X) = k+1$, we have $(w')^{T} (1^{n} , X) \neq 0$. 
But $(w')^{T} X = 0$. Therefore, $c := (w')^{T} 1^{n} \neq 0$. Take $w := c^{-1} w'$. Then $w^{T} 1^{n} = 1$ but $w^{T} X = 0$. 

Therefore, $Y - 1^{n} w^{T} Y = (X,y) - 1^{n} (w^{T} X, w^{T} y) 
= \bigl( X, y - 1^{n} (w^{T} y) \bigr)$. 
Since $w^{T} 1^{n} = 1$ and $Y \in \Pf^{k}$, by \eqref{E:when.is.Y.in.Pk}, 
$rank \, (Y - 1^{n} w^{T} Y) = k$. Thus, $rank \, \bigl( X , y - 1^{n} w^{T} y \bigr) = k$. But $rank \, X = k$. Hence, there exists a unique $b_{1}^{k \times 1}$ s.t.\ 
$y - (w^{T} y) 1^{n} - X b_{1} = 0$. 
I.e., $L^{1}(b) = 0$, where $b^{(k+1) \times 1}  := (w^{T} y, b_{1}^{T})^{T}$. 
Therefore, $\beta(Y) = b^{T}$ uniquely, 
so $Y \in \Y_{LAD}'$, as claimed. This proves the claim \eqref{E:noncollin.in.Pf.is.in.Y.LAD'}. 

Since $\Phi_{LAD}$ is continuous on $\Y_{LAD}'$, we have
	\begin{equation}   \label{E:not.coll.in.Pk.not.LAD.sing}
		\text{No noncollinear data set in } \Pf^{k}  \text{ is a singularity of } \Phi_{LAD}.
	\end{equation}
It follows from remark \ref{R:perf.fits.in.reg} that LAD is a plane-fitting method.

By lemma \ref{L:diff.rank.condn.implications}\eqref{I:diff.rank.condn.implies.not.LAD.sing} in appendix \ref{Chptr:LAD.technicalities}, we have that if 
$Y = (X^{n \times k}, y^{n \times 1}) \in \Y$ is collinear but the following holds then $Y$ is \emph{not} a singularity w.r.t.\ $\Y'_{LAD}$.
 	\begin{multline}    \label{E:diff.rank.condition}   
		\text{If } 1 \leq i_{1} < \cdots < i_{\nvar} \leq n, \text{ then } \\ 
			(x_{i_{2}} - x_{i_{1}}, y_{i_{2}} - y_{i_{1}}), 
			           \cdots, (x_{i_{\nvar}} - x_{i_{1}}, y_{i_{\nvar}} - y_{i_{1}}) 
				\text{ are linearly independent.}
 	\end{multline}
($\nvar = k+1$.) From corollary \ref{C:defective.collin.data.sets.dim}
we see that almost all collinear data sets satisfy condition \eqref{E:diff.rank.condition} and hence are \emph{not} singularities of LAD. Example \ref{Ex:collnr.not.sing.of.LAD} displays a class of collinear data sets satisfying \eqref{E:diff.rank.condition}. 
As observed in subsection \ref{SS:functions.vs.geometry}, LAD is an example of a regression method that, when viewed as a function-valued map, has more singularities than it does when viewed as a plane-valued map. 

See appendix \ref{Chptr:misc.proofs} for the proof of the following. 

	\begin{lemma}  \label{L:power.diff.lin.indep}
Let $n = 2, 3, \ldots$ and let $z = (z_{1}, \ldots, z_{n})$ be a vector of distinct, possibly complex, nonzero numbers. Let $\ell = 2, 3, \ldots, n$ and let $Z^{n \times \ell}$ be the matrix whose $i^{th}$ row is $w_{i} := (z_{i}^{0}, z_{i}^{1}, \ldots, z_{i}^{\ell-1})$. (This time superscripts are exponents.)  Then for any $1 \leq i_{1} < \ldots < i_{\ell} \leq n$ the vectors $w_{i_{2}} - w_{i_{1}}, \ldots, w_{i_{\ell}} - w_{i_{1}}$ are linearly independent (over the complex numbers, $\mathbb{C}$).
	\end{lemma}

   \begin{example}[Nonsingular collinear data sets] \label{Ex:collnr.not.sing.of.LAD}
Assume \eqref{E:n>nvar>k>0} holds. Let $z_{1}, \ldots, z_{n}$ be distinct nonzero real numbers. 
Let $x_{i} = (z_{i}^{0}, z_{i}^{1}, \ldots, z_{i}^{k-1})$ and $y_{i} = z_{i}^{k}$ ($i=1, \ldots, n$). 
Let $Y_{0} = \bigl( (x_{1}, y_{1})^{T}, \ldots, (x_{n}, y_{n})^{T} \bigr)^{T}$. Then $Y_{0}$ is collinear (definition \ref{D:collinearity}) because for every $i = 2, \ldots, n$, the first coordinate of the $k$-vector $x_{i} - x_{1}$ is 0. Yet, by lemma \ref{L:power.diff.lin.indep} 
with $\ell = \nvar = k+1$, condition \eqref{E:diff.rank.condition} holds. Hence, by lemma \ref{L:diff.rank.condn.implications}, $Y_{0}$ is not a singularity of LAD.
   \end{example} 

\subsection{Codimension of singular set of LAD}  \label{SSS:codim.LAD.sings}
First, we show that we may apply proposition \ref{P:sing.codim.in.plane-fitting} to show that \eqref{E:big.lwr.bound.in.plane.fit} holds for LAD.
All that needs to be proved is that \textbf{hypothesis \ref{Hyp:S.cap.T.small}} of theorem \ref{T:Phi.star.Hr.contains.Theta.star.Hr} holds. By proposition \ref{P:few.collin.LAD.sings}, to prove that hypothesis \ref{Hyp:S.cap.T.small} holds it suffices to consider the behavior of LAD on collinear data sets. 
    \begin{equation*}
      \text{Let } Y_{0} = (X,y) \text{ be a collinear data set satisfying condition 
        \eqref{E:diff.rank.condition}.}
    \end{equation*}
Then, by lemma \ref{L:diff.rank.condn.implications}, 
	\begin{equation}  \label{E:LAD.y.aint.0}
		y \neq 0.
	\end{equation}
Neither the collinearity (definitiion \ref{D:collinearity}) nor condition \eqref{E:diff.rank.condition} are invalidated if we mean center $X$ (remark \ref{R:mean.center}) and $y$. So we may assume $1_{n} X = 0^{n \times 1}$ and 
$1_{n} y = 0^{1 \times 1}$. I.e.,
	\begin{equation}  \label{E:1n.Y0.=.0}
		1_{n} Y_{0} = 0^{1 \times \nvar}.
	\end{equation}
By lemma \ref{L:diff.rank.condn.implications},  
$\text{rank} \, Y_{0} = k$, $Y_{0} \in \Pf^{k}$, and the rows of $Y_{0}$ lie exactly on a unique plane $\xi \in G(k,\nvar)$ passing through the origin. I.e., $\xi$ is the row space of $Y_{0}$. By lemma \ref{L:diff.rank.condn.implications}\eqref{I:X.mean.centered.rank}, 
    \begin{equation} \label{E:rank.X=k-1}
      rank \, X = k-1 .
    \end{equation}
Hence, there is a unit vector,  $z^{1 \times k}$, unique up to sign, orthogonal to the row space of $X$:
    \begin{equation}  \label{E:|z|=1.z.perp.rho(X)}
      |z| = 1 \text{ and } z \perp \rho(X) .
    \end{equation}
Thus, since $rank \, Y_{0} = k$, we have that $(z, 0^{1 \times 1})$ is, up to sign, the unique unit vector orthogonal to $\xi = \rho(Y_{0})$. (In this section, $\nvar = k+1$.) It follows that 
    \begin{equation}  \label{E:(z,0).and.rows.of.Y0}
      (z, 0) \text{  and the rows of $Y_{0}$ span } \RR^{\nvar} .
    \end{equation} 

Since $n > \nvar > k$, by \eqref{E:n>nvar>k>0}, the orthogonal complement, $1_{n}$, 
of $1^{n}$ in $\RR^{n}$ has dimension $n -1 > k = \text{rank} \, Y_{0}$. By \eqref{E:1n.Y0.=.0}, the column space, $C$, of $Y_{0}$ lies in $1_{n}$. Therefore, there exists a unit vector 
$g^{n \times 1} \in 1_{n}$ s.t.\ $g \perp C$. I.e.,
    \begin{equation}  \label{E:g.Y0.and.g.1n.=.0}
	g^{T} Y_{0} = 0 \text{ and } 1_{n} \, g = 0 .
    \end{equation}
Let 
    \begin{equation} \label{E:Z.=.gz}
      Z^{n \times k} :=  g z .
    \end{equation}
Let 
	\begin{equation} \label{E:Y.X.Z.mat}
	   \mbf{Y}^{n \times \nvar} := (X + Z, y) = Y_{0} + (Z, 0)^{n \times \nvar}.
	\end{equation}
By \eqref{E:1n.Y0.=.0} and \eqref{E:g.Y0.and.g.1n.=.0}, $1_{n} \mbf{Y} = 0$. \emph{Claim:} 
	\begin{equation}   \label{E:Y.has.full.rank}
	  \mbf{Y} \text{ is of full rank, } \nvar. 
	\end{equation}
For suppose $\mbf{Y}$ is not of full rank. Then there exists $u^{1 \times \nvar} \neq 0$ s.t.\ 
$0 = \mbf{Y} u^{T}$. From \eqref{E:(z,0).and.rows.of.Y0}, we may assume that, 
for some $\gamma \in \RR$ and $b \in \rho(Y_{0})$, 
we have $u = b + \gamma \, (z, 0)$. Now, by \eqref{E:|z|=1.z.perp.rho(X)} and choice of $b$ we have, $Y_{0} \,  (z, 0)^{T} = 0$ and $(Z, 0) \, b^{T} = 0$. Hence, 
	\begin{equation} \label{E:0.Y0.w.z}
		0^{n \times 1} = \mbf{Y} \, u^{T} = Y_{0} b^{T} +  \gamma \, (Z, 0) (z, 0)^{T} 
		         = Y_{0} b^{T} +  \gamma \, g \, (z,0) (z, 0)^{T} 
			= Y_{0} b^{T} +  \gamma \, g.
	\end{equation}
Hence, by \eqref{E:g.Y0.and.g.1n.=.0}, 
	\[
		0 = g^{T} \bigl( Y_{0} b^{T} +  \gamma \, g \bigr) =   \gamma.
	\]
I.e., $\gamma = 0$ so $u = b  \in \rho(Y_{0})$. Therefore, there exists $v^{n \times 1}$ s.t.\ $v^{T} Y_{0} = b$. Therefore, from \eqref{E:0.Y0.w.z} and the fact that $\gamma = 0$, 
	\[
		0 = Y_{0} b^{T} = Y_{0} Y_{0}^{T} v \quad \text{  so  } 
		  \quad 0 = v^{T} Y_{0} Y_{0}^{T} v = |b|^{2}.
	\]
I.e., $b = 0$. Therefore, $u = b = 0$, a contradiction, and the claim \eqref{E:Y.has.full.rank} is proved. We have already observed that $1_{n} \mbf{Y} = 0$, so \eqref{E:plane.fitting.wT.Y=0} holds. Thus, we may use $\mbf{Y}$ in the construction in section \ref{SS:D.T.plane.fit}, 

Now let $v_{1} \in \RR^{\nvar}$ in section \ref{SS:D.T.plane.fit} just be the unit vector 
$(z,0)^{1 \times \nvar}$ so $v_{1}$ is orthogonal to $\xi = \rho(Y_{0}) \in G(k, \nvar)$ and, by \eqref{E:Z.=.gz}, $(Z,0) = g v_{1}$. In fact, we have seen that $v_{1} = (z,0)$ spans 
$\xi^{\perp}$. (See \eqref{E:superscript.perp.notation} and \eqref{E:(z,0).and.rows.of.Y0}.) Let $v_{2} := (0^{1 \times k}, 1) \perp v_{1}$. Thus,
	\begin{equation}  \label{E:Y0.X.y.v2}
		Y_{0} = (X, 0)^{n \times \nvar} + y^{1 \times n} v_{2}.
	\end{equation}

Since $\xi^{\perp}$ is the line spanned by $v_{1}$, we have $v_{2} \in \xi$. Let 
$\zeta \in G(k-1,\nvar)$ be the $(k-1)$-dimensional subspace of $\xi$ perpendicular to $v_{2}$. (So $\zeta \perp v_{1}$ as well.) Since $\text{rank} \, X = k-1$, the space 
$\zeta$ is just the row space of $(X, 0^{n \times 1})$. 

If $\ell \in P^{1}$, let $(c,s)^{1 \times 2} \in \ell$ be a unit vector and let $\lambda(\ell)$ and 
$\Upsilon$ be defined as in \eqref{E:lambda(ell).defn} and \eqref{E:plane-fitting.Upsilon.defn}, resp., with $\mbf{Y}$ defined by \eqref{E:Y.X.Z.mat}. 
Let $\pi^{\nvar \times \nvar}$ be the matrix of orthogonal projection of $\RR^{\nvar}$ onto the $(k-1)$-dimensional space $\zeta$, the row space of $(X, 0^{n \times 1})$. Then the matrix of orthogonal projection onto $\lambda(\ell)$ is
	\begin{equation*}
		\Pi \bigl[ \lambda(\ell) \bigr] := \pi + (c v_{2} + s v_{1})^{T} (c v_{2} + s v_{1}).
	\end{equation*}
Hence, by \eqref{E:Y.X.Z.mat}, \eqref{E:Y0.X.y.v2}, \eqref{E:Z.=.gz},  \eqref{E:g.Y0.and.g.1n.=.0}, and recalling that $v_{1} = (z,0)^{1 \times \nvar}$, we have
	\begin{align*}
	  \Upsilon(\ell) &:= \mbf{Y} \, \Pi \bigl[ \lambda(\ell) \bigr] \\
	            &= \bigl[ (X, 0)^{n \times \nvar} 
			   + y^{1 \times n} v_{2} + g v_{1} \bigr] \Pi \bigl[ \lambda(\ell) \bigr] \\
		    &= (X,0) + (c y + s g) (c v_{2} + s v_{1})  \\
		    &= \bigl( X + (cs \, y + s^{2} g) z, c^{2} y + sc g \bigr) .
	\end{align*}

By \eqref{E:Upsilon.maps.P1.into.Pf}, 
    \begin{equation*}
      \Upsilon(\ell) \in \Pf^{k} .
    \end{equation*}
Thus, when $s = 0$ and $c = \pm 1$, then $\Upsilon(\ell) = Y_{0}$ and so is collinear. 
($s = 0$ and $c = \pm 1$ corresponds 
to just one $\ell \in P^{1}$.) \emph{Claim:} If $s \neq 0$, then $\Upsilon(\ell)$ is not collinear. By lemma \ref{L:Y.collnr.iff.rank.X1.<.k+1}, it suffices to show that the matrix
	\begin{equation*}
		X_{1}(\ell) := \bigl( 1^{n}, X + (cs \, y + s^{2} g) z \bigr)
	\end{equation*}
has rank $\nvar = k+1$ when $s \neq 0$. 

Let $W^{n \times k} := X + (cs \, y + s^{2} g) z$. Let $a^{1 \times k} \neq 0$. We show $W a^{T} \neq 0$ if $s \neq 0$. By \eqref{E:rank.X=k-1}, 
$\text{rank} \, X = k-1$. Moreover, $z$ spans the orthogonal complement, 
$\rho(X)^{\perp}$, of  the row space, $\rho(X)$, of $X$. Hence, $a = x + \lambda z$, for some $x \in \rho(X)$ and $\lambda \in \RR$.
Suppose $s \neq 0$, but $W a^{T} = 0$. Then, by \eqref{E:g.Y0.and.g.1n.=.0}, 
$0 = g^{T} W a^{T} =  s^{2} z a^{T} = s^{2} \lambda$. 
Hence, $\lambda = 0$, so $a = x \in \rho(X)$. Thus, we can write $a = u X$, where 
$u^{1 \times n} \neq 0$ is some $1 \times n$ row vector, and $W a^{T} = X a^{T}$, 
by \eqref{E:|z|=1.z.perp.rho(X)}. Hence, 
$0  = u W a^{T}= u X a^{T} = (u X) (X u)^{T} = |a|^{2}$. I.e., $a = 0$, contradiction. 
Hence, $W a^{T} \neq 0$, as desired. Therefore, $X + (cs \, y + s^{2} g) z$ has rank $k$. 

Moreover, by \eqref{E:1n.Y0.=.0}, and \eqref{E:g.Y0.and.g.1n.=.0}, 
$1_{n} \bigl[ X + (cs \, y + s^{2} g) z ] = 0$. Thus, $1^{n}$ is not in the column space of $W$. Let $W_{1} := (1^{n}, W)$. It follows that 
$rank \, X_{1}(\ell)  = rank \, W_{1}  = \nvar$ if $s \neq 0$, as desired. This completes the proof of the claim that $\Upsilon(\ell)$ is not collinear.

Hence, if $s \neq 0$ then $\Upsilon(\ell)$ is not collinear. Therefore, by the fact that $\Upsilon(\ell) \in \Pf^{k}$ and \eqref{E:not.coll.in.Pk.not.LAD.sing}, if $s \neq 0$ then $\Upsilon(\ell)$ is not a singularity of LAD.

Recall that in section \ref{SS:D.T.plane.fit} two versions of $\D$ are proposed, $\D_{\mu}$ and $\D_{\infty}$. First, consider the $\D = \D_{\infty}$ case. 
Let $\ell_{0} \in P^{1}$ be the span of 
$(1,0) \in \RR^{2}$. $\ell_{0}$ corresponds to $s = 0$. We have seen that $\Upsilon(\ell)$ is collinear precisely when $\ell = \ell_{0}$. In fact, $\Upsilon(\ell_{0}) = Y_{0}$. By \eqref{E:T.infty:=PS[Upsilon(P1)]}, $\T_{\infty}$ is identified with $\Upsilon(P^{1})$.
By assumption, $Y_{0}$, the only collinear data set in $\T_{\infty}$, satisfies \eqref{E:diff.rank.condition}. Therefore, by lemma \ref{L:diff.rank.condn.implications}, $Y_{0}$ is not a singularity of LAD. Thus, the space 
$\T_{\infty}$ does not include any singularities of LAD. I.e., LAD satisfies \textbf{hypothesis \ref{Hyp:S.cap.T.small}} of theorem \ref{T:Phi.star.Hr.contains.Theta.star.Hr} with 
$\T = \T_{\infty}$.

Next, consider the $\D = \D_{\mu}$ case. $\Upsilon(\ell)$ is collinear precisely when 
$\ell = \ell_{0}$ and $\Upsilon(\ell_{0}) = Y_{0}$. Now, $Y \in \Y$ is collinear, if and only if 
$s Y$ is collinear for any $s \neq 0$. Hence, by \eqref{E:R.mu.defn}, 
$R_{\mu} \circ \Upsilon(\ell)$ is collinear precisely when $\ell = \ell_{0}$. 
Hence, by \eqref{E:plane.fittin.T.mu.defn} and \eqref{E:Upsilon.Delta.mu.defn}, the only collinear data set in $\T_{\mu}$ is $R_{\mu}(Y_{0}) = R_{\mu} \circ \Upsilon(\ell_{0}) 
= \Upsilon_{\mu}(\ell_{0}) \in \T_{\mu}$.  

By assumption, $Y_{0}$ satisfies \eqref{E:diff.rank.condition} and, like collinearity, the property of satisfying \eqref{E:diff.rank.condition} is invariant under rescaling. Hence, $R_{\mu}(Y_{0})$, the only collinear data set in $\T_{\mu}$, satisfies \eqref{E:diff.rank.condition}. Therefore, by lemma \ref{L:diff.rank.condn.implications} again, $R_{\mu}(Y_{0})$ is not a singularity of LAD. Thus, the space 
$\T_{\mu}$ does not include any singularities of LAD. I.e., LAD satisfies \textbf{hypothesis \ref{Hyp:S.cap.T.small}} of theorem \ref{T:Phi.star.Hr.contains.Theta.star.Hr} with 
$\T = \T_{\mu}$. 

Hence, whether $\D = \D_{\mu}$ or $\D_{\infty}$, proposition 
\ref{P:sing.codim.in.plane-fitting} holds for LAD. Therefore, 
    \begin{multline}  \label{E:codim.LAD.sing.set}
      \text{Let } \msf{V} \text{ be a cover of } G(k,k+1) 
        \text{ on which a convex combination function is defined.} \\
          \text{Then the codimension of the set of } 
          \msf{V}-\text{severe singularities of LAD } \\
            \text{is no greater than 2 and \eqref{E:Hm.a.S.lwr.bnd.for.plane.fitting.Pf} 
            holds with } a = \dim \D - 2 .
    \end{multline}
    
(But see proposition \ref{P:codim.S.LAD.=1.n-k.even}.) As we will see in section \ref{SS:lin.combo.for.plane.fit.w/.k=nvar-1}, $\msf{V}$ can be chosen so that $\msf{V}$-severity is quite severe. Note that, if the $X$ matrix is fixed 
and $Y = (X,y)$ is not collinear, then $\Phi_{LAD}(X,y)$ is continuous in $y^{n \times 1}$, indeed Lipschitz (\cite{spE95.smoothL1}). Hence, at noncollinear singularities, LAD can only be hypersensitive to perturbations that perturb $X$. That is the case in figure \ref{F:HeightFen}. (See remark \ref{R:robust.lin.reg}.)

By \eqref{E:codim.LAD.sing.set}, the codimension of the singular set of LAD is no greater than 2. But by lemma \ref{E:exact.fit.basis}, if $Y \in \Y_{LAD}'$ then on its unique LAD plane lie at least $\nvar$ observations, i.e. rows of $Y$. Thus, for each choice of 
$0 \leq i_{1} < \cdots < i_{\nvar} \leq n$ there is a region of $\Y_{LAD}'$ in which the fitted LAD plane contains rows $i_{1} < \cdots < i_{\nvar}$ of $Y$ and only those rows. Must the noncollinear singularities of $LAD$ lie on the boundaries of these regions? Anyway, if that is true it suggests that the set of singularities of LAD is one-dimensional. At least half the time it is (see appendix \ref{Chptr:misc.proofs} for proof): 
  \begin{prop}  \label{P:codim.S.LAD.=1.n-k.even}
When $n-k$ is even the singular set of LAD has codimension 1.
  \end{prop}

The proposition is consistent with the bottom left panel in figure \ref{F:LS.PC.LAD.lf.plots} where we see three lines of singularities of LAD when $n = 3$, $k=1$. It is proved in the appendix of \cite{spE.3.or.4} that $codim \, \Ss_{LAD} = 1$ when $n=4$ and $k=1$. Figure \ref{F:sing.dists.cdfs} appears consistent with that. I conjecture that in general even if $n-k$ is odd it is still the case that $codim \, \Ss_{LAD} = 1$. I leave the determination of the truth value of that conjecture as an exercise for the reader.
 
I also conjecture that the codimension of the set of $90^{\circ}$ singularities (section \ref{SS:lin.combo.for.plane.fit.w/.k=nvar-1}) of LAD is 2. I believe that the only $90^{\circ}$ singularity of LAD in figure \ref{F:LS.PC.LAD.lf.plots} is the one shown in the panel ``(LAD,c)''. (I successfully fought off the temptation to prove that.) If so, that is consistent with this conjecture.

The difficulty of all this points up the usefulness of proposition 
\ref{P:sing.codim.in.plane-fitting}.

  \begin{remark}[Robust linear regression]  \label{R:robust.lin.reg}
LAD is often recommended because it is more resistant to ``outliers'' (extreme data points, specifically in this case extreme components of $y$) than is LS (Bloomfield and Steiger \cite[Section 2.3]{pBwlS83}). 

An even more resistant linear regression method is ``least median of squares'' regression (LMS; Hampel \cite[p.\ 380]{frH75.robust}, Rousseeuw \cite{pjR84.LMS}, 
Rousseeuw and Leroy \cite[p.\ 14]{pjRamL03.robust}). 
By \cite[Theorem 3.1, p.\ 343]{spE98}, 
the singular set of LMS has codimension no greater than 2. (Recall \eqref{E:integer.part.floor}. For LMS assume 
$\lfloor n/2 \rfloor > k = \nvar-1$, where $\lfloor n/2 \rfloor$ is the largest integer $\leq n/2$.) By \eqref{E:codim.LAD.sing.set}, the singular set of LMS has codimension no greater than 2 as well. (See proposition \ref{P:codim.S.LAD.=1.n-k.even} and remark \ref{R:penalty.for.robustness.in.lin.reg?}.)

\cite{spE00} argues that sensitivity to outliers is just ``singularity at infinity" and an overall measure of stability of a regression method is the codimension of the ``extended singular set'', consisting of both ordinary singularities, the main topic of this chapter, and singularities at infinity. By this measure, LS is at least as stable as LAD or LMS.  
The sets of the ordinary singularities of LAD and LMS is at least as big in dimension as the extended singular set of LS, bigger if $n - k > 2$ and $k > 1$.

It would be interesting if the set of singularities at infinity could be usefully studied using, 
say, proposition \ref{P:sing.dim.when.H.d-r.D.=.0}.
  \end{remark}

Let $\alpha > 0$ and consider the general problem
	\[
		\text{Find $b$ to minimize } \| y - X b \|_{\alpha},
	\]
where $\| v \|_{\alpha}$ is the $L^{\alpha}$ norm of the vector $v$. We know that this operation has 
a singular set of dimension at least $n \nvar - 2$ if $\alpha = 1$ and of dimension $\dim \Pf -1 = (n+1)k$ if $\alpha = 2$ (proposition \ref{P:collin.data.are.sings.of.LS}, lemma \ref{L:dim.set.of.collin.data.sets}). An interesting question is, what is the supremum of the set of $\alpha$ for which the codimension is no greater than 2?  It would be interesting if that supremum were $> 1$.

  \begin{remark}[Finer covers]
Let $\phi \in (0, \pi/4]$ and, for $\xi_{0} \in P(S)$, define $V_{\phi}(\xi_{0})$ as in the definition, \eqref{E:V.xi0.defn}, 
of $V(\xi_{0})$ but with \eqref{E:W(L0).defn} replaced by $v \cdot v_{0} > \cos \phi$. 
Then $V_{\phi}(\xi_{0}) \subset V(\xi_{0})$ and $V(\xi_{0}) = V_{\pi/4}(\xi_{0})$. 
The finer covers 
$\msf{V}_{\phi} := \{ V_{\phi}(\xi) \subset \F : \xi \in P(S) \}$ ($0 < \phi \leq \pi/4$) might be useful in applications to linear regression of proposition \ref{P:V1.V2.V.homotop}. 
  \end{remark}

\section{A convex combination function for plane fitting with $k = \nvar-1$}  \label{SS:lin.combo.for.plane.fit.w/.k=nvar-1} 
Since the Grassmann manifold, $G(k,\nvar)$ is a Riemannian manifold, proposition \ref{P:smooth.manifs.have.commutative.convex.combos}  
tells us that one can construct a commutative convex combination function (definition \ref{D:convex.combo.fn}) on a cover of $\F = G(k,\nvar)$ consisting or open geodesic balls of sufficiently small radius.  

Proposition \ref{P:smooth.manifs.have.convex.combos} tells us that for any open cover of $\F$ consisting of geodesically convex subsets, one can at least construct a possibly noncommutative convex combination function. Jost and Xin \cite{jJylX1999.Bernstein.theorems} shows how to construct geodesically convex neighborhoods in a Grassmann manifold.  
See also H\"{u}per \emph{et al} \cite{kHuHsH2010.MeansOnGrassmannians}.

In this section we take
    \begin{equation}  \label{E:nvar=k+1}
      \nvar = k+1 ,
    \end{equation}
the usual linear regression case, but we do not confine ourselves to regression methods here. In this setting it is easy to describe a method for constructing convex combinations. Here, one can give an explicit sufficient condition for convex combinations of planes to make sense. That is because, by Milnor and Stasheff \cite[lemma 5.1 and remark following it, p.\ 57]{jwMjdS74}, $G(\nvar-1, \nvar)$ is homeomorphic to $G(1,\nvar)$ which is just the $k$-dimensional projective space, $P^{k}$. So if $k = \nvar-1$, taking convex combinations of planes is equivalent to taking convex combinations of lines in $P^{k}$. So here we develop a convex combination function for projective space. Convex combinations of lines are constructed in \cite{spE91.top.direct.axis}. The method we develop here is essentially the same. 

For study of non-parametric regression (remark \ref{R:nonparam.reg}), for example, it might be helpful to consider a more general setup. Let $S$ be an inner product space 
over $\RR$. $S$ might be infinite dimensional, but does not have to be a Hilbert space, i.e., complete. If $v,w \in S$, write the inner product and norm 
as $v \cdot w$ 
and $|v| := \sqrt{v \cdot v}$, respectively. Let $S_{0} := S \setminus \{0\}$. 
If $v,w \in S_{0}$, say that $v$ and $w$ are \emph{equivalent} and write $v \sim w$ if there exists $\alpha \in \RR$ s.t.\ $v = \alpha w$. Thus, equivalence classes are lines in $S$ through the origin. 
    \begin{equation} \label{E:proj.space.defn}
      \text{Let the projective space, } P(S), \text{ be the space } 
        S_{0} / \sim \text{ of equivalence classes.}
    \end{equation} 
Give $P(S)$ the quotient topology (Munkres \cite[p.\ 112]{jrM84}). In this section we initially take $\F := P(S)$. If $L \in  P(S)$, say that $L$ is ``oriented'' if a choice has been made of a unit vector $o(L) \in L$. (This use of ``o'' is not Landau notation. Strictly speaking an element of $P(S)$ is a one-dimensional subspace of $S$, a line through the origin, with the origin plucked out, but we will not be fussy about whether the origin is included or not.)

If $x, y \in S$ are nonzero vectors, define the angle between them to be 
    \begin{equation}  \label{E:angle.between.vectors}
      \angle (x,y) := \arccos \bigl( |x|^{-1} |y|^{-1} ( x \cdot y ) \bigr) \in [0, \pi] .
    \end{equation} 
Notice that
	\begin{multline} \label{E:angle.-x,y}
		\angle(- x,y) = \pi - \angle( x,y) . \text{ Consequently, }
		  \angle(- x,-y) = \pi - \angle( x,-y) \\
		    = \pi - \bigl( \pi - \angle( x,y) \bigr) = \angle( x,y) .
	\end{multline}

  \begin{remark}  \label{R:metrics.on.P(S).and.G(k,k+1)}
Define a metric on $P(S)$ as follows. Let $L, M \in P(S)$. Then one can orient $L, M$ by choosing unit vectors $o(L) \in L$ and $o(M) \in M$ s.t.\  
$o(L) \cdot o(M) \geq 0$. Now define the distance between $L, M$ to just be the angle between $o(L), o(M)$. Call that distance the angle between $L$ and $M$ and write 
    \begin{equation}  \label{E:angle.L.M}
      \angle (L, M) := \angle \bigl[ o(L) , o(M) \bigr] 
        = \arccos \bigl[ o(L) \cdot o(M) \bigr] \in [0, \pi/2] .
    \end{equation}
    
Now let 
    \begin{equation*}
      S = \RR^{\nvar} .
    \end{equation*} 
We prove $(L, M) \mapsto \angle (L,M)$ is a metric on $P(\RR^{\nvar})$. If $L,M \in P(\RR^{\nvar})$, then $o(L), o(M) \in S^{\nvar-1}$ (the $(\nvar-1)$-sphere). Now, $\angle$ gives geodesic distance on $S^{\nvar-1}$. (Recall that 
$\angle$ is the shorter great circle distance on the sphere, $S^{\nvar-1}$.) Thus, by Boothby \cite[Theorem (3.1), p.\ 187 and Corollary (7.11), p.\ 346]{wmB75}, 
$\angle$, defined by \eqref{E:angle.between.vectors} and restricted to $S^{\nvar-1}$, is a metric on $S^{\nvar-1}$.

We show that $\angle : P(\RR^{\nvar}) \to [0, \infty)$ metrizes the quotient topology on $P(\RR^{\nvar})$. Let $p : S_{0} \to P(\RR^{\nvar})$ be the quotient map. (Do not confuse this ``$p$'' with ``$p := \dim \Pf$''.) Thus, if $x \in S_{0}$, then $p(x)$ is the line through the origin in $\RR^{\nvar}$ containing $x$. The map $h : S_{0} \times S_{0} \to [0,\infty)$ defined by 
$h(x,y) := \arccos \bigl( |x|^{-1}|y|^{-1} | x \cdot y | \bigr)$ is continuous. 
Moreover, if $x', y' \in S_{0}$ with $x' \sim x$, $y' \sim y$ then $h(x',y') = h(x,y)$. It follows from Munkres \cite[p.\ 112]{jrM84} that $h$ induces a continuous map 
$g : P(\RR^{\nvar})^{2} \to [0, \infty)$. But $g = \angle$. Thus, the topology generated by $\angle$, interpreted as a metric on $P(\RR^{\nvar})$, is no finer than the quotient topology.

Conversely, let $U \subset P(\RR^{\nvar})$ be open in the quotient topology and let $L \in U$. 
Let $\mbf{B}_{s}(L) \subset P(\RR^{\nvar})$ be the open ball about $L$ with radius $s > 0$ as measured by $\angle$. We show that for some 
$s > 0$, $\mbf{B}_{s}(L) \subset U$. This means that the quotient topology is no finer than 
the $\angle$ topology. Let $V := p^{-1}(U) \subset S_{0}$. Then $\alpha V = V$ 
for every $\alpha \in \RR \setminus \{0\}$ and 
$L \subset V$. Let $x \in L \cap S_{0}$. Then $x \in V$, in particular $x \neq 0$, and 
$\alpha x \in L \subset V$ for every $\alpha \in \RR \setminus \{0\}$. 
$V$ is open in $S_{0}$, which means $V$ is open in $S = \RR^{\nvar}$. Hence, there exists 
$r > 0$ s.t.\ the open ball $B_{r}(x) \subset V$, where $B_{r}(x)$ is the open ball in 
$\RR^{\nvar}$ corresponding to the norm $| \cdot |$ on $\RR^{\nvar}$. Thus, $0 \notin B_{r}(x)$, so $r < |x|$. 
Let $y \in B_{r}(x)$ satisfy $|y| = |x|$. Then $r^{2} > |x - y|^{2} = 2|x|^{2} - 2 x \cdot y$. In particular, $2 x \cdot y > 2|x|^{2} - r^{2} > 0$. I.e., $x \cdot y > 0$. Thus, 
    \begin{equation*}  \label{E:x.dot.y/|x|.sqrd.ineq}
       \cos \angle \bigl( L, p(y) \bigr)
         = \bigl| \cos \angle ( x, y ) \bigr| = \frac{|x \cdot y|}{|x||y|} 
           = \frac{x \cdot y}{|x|^{2}} > 1 - \frac{r^{2}}{2 |x|^{2}} > \frac{1}{2} .
    \end{equation*}
Conversely, with $|x| = |y|$, 
    \begin{equation}  \label{E:x.dot.y/|x|.sqrd.ineq}
        \cos \angle ( x, y ) > 1 - \frac{r^{2}}{2 |x|^{2}} \text{ and } x \cdot y \geq 0
           \text{ imply } |x - y|^{2} < r^{2} .
    \end{equation}
I.e., if $\cos \angle ( x, y ) > 1 - \frac{r^{2}}{2 |x|^{2}}$,  $x \cdot y \geq 0$, and $|x| = |y|$, then $y \in B_{r}(x) \subset V$.

Let $s := \arccos \bigl( 1 - r^{2}/(2 |x|^{2}) \bigr) < \arccos(1/2)$. Suppose $y \in S_{0}$, 
$|y| = |x|$, and $\bigl| \cos \angle(x,y) \bigr| > \cos s$. If $x \cdot y \geq 0$ then 
by \eqref{E:x.dot.y/|x|.sqrd.ineq}, $y \in B_{r}(x) \subset V$. Therefore, $\alpha y \in V$ for every $\alpha \in \RR \setminus \{0\}$. I.e., $p(y) \in U$. 

Let $L' \in \mbf{B}_{s}(L)$, so $p^{-1}(L') \subset p^{-1} \bigl( \mbf{B}_{s}(L) \bigr)$. 
Now, $y \in p^{-1}(L')$ if and only if $-y \in p^{-1}(L')$. Hence, we may assume $x \cdot y \geq 0$. Also $y \in p^{-1}(L')$ if and only if $(|x|/|y|) y \in p^{-1}(L')$. 
But $\bigl| (|x|/|y|) \, y \bigr| = |x|$ and $L' \in \mbf{B}_{s}(L)$ means 
$\angle(L',L) = \arccos \bigl| \cos \angle \bigl( x, (|x|/|y|) y \bigr) \bigr| < s$. 
I.e., $\bigl| \cos \angle \bigl( x, (|x|/|y|) y \bigr) \bigr| > \cos s$.
Therefore, $(|x|/|y|) \, y \in V$, hence $y \in V$. That means so $L' \in U$. 
I.e., $\mbf{B}_{s}(L) \subset U$. 

This proves that $\angle$ generates the quotient topology on $P(\RR^{\nvar})$. This is more or less a special case of Wong \cite[Theorem 8(a), p. 591]{ycW67.Grassmanif}.

Since, $P(\RR^{\nvar})$ and $G(k, k+1)$ are homeomorphic (Milnor and Stasheff \cite[Lemma 5.1, p.\ 57]{jwMjdS74}; recall $k+1 = \nvar$), we can define a metric on $G(k, k+1)$ as follows. If $\xi, \zeta \in G(k, k+1)$, let $L, M \in P^{k} = P(\RR^{k+1})$ be the unique lines (one-dimensional subspaces) in $\RR^{k+1}$ orthogonal to $\xi$, $\zeta$, resp., and define 
    \begin{equation}  \label{E:metric.on.G(k,k+1)}
      d(\xi, \zeta) := \angle(L,M) , 
    \end{equation}
where the latter is defined in \eqref{E:angle.L.M}.
  \end{remark}

Define:
    \begin{multline}  \label{E:sign.function}
      \text{If } t \in \RR, \; sign(t) = sign \, t := \pm1 \text{ making } sign(t) \, t \geq 0. \\
        \text{ (Define } sign(0) = \pm 1, \text{ whichever is convenient at the time.)} 
    \end{multline}

Take
    \begin{equation*}
      S := \RR^{\nvar} .
    \end{equation*}
Let $V \subset P(\RR^{\nvar})$ and let $M \in P(\RR^{\nvar})$. Let $o(M)$ be a unit vector in the line $M$. We have 
    \begin{equation} \label{E:can.orient.leq.pi/2}
      \text{All the lines in $V$ can be oriented so that }
        \angle \bigl[ o(L), o(M) \bigr] \leq \pi/2 
          \text{ for every } L \in V.
    \end{equation}
With this orientation $o(L) \cdot o(M) \geq 0$ for every $L \in V$. To prove \eqref{E:can.orient.leq.pi/2}, let $L \in V$. If $v$ is a unit vector in $L$. Let $o(L) := sign \bigl( v \cdot o(M) \bigr) v$. If $v \cdot o(M) = 0$ define $sign \bigl( v \cdot o(M) \bigr) = \pm 1$ arbitrarily. (See \eqref{E:sign.function}.)
 
Suppose one can orient all the lines in $V$ in such a way 
that $L, M  \in  V$ implies $o(L) \cdot  o(M ) > 0$. 
Thus, if $L, M \in V$ then $L, M$ are ``within 90 degrees'' of each other.
If this holds say that $V$ is ``acute.''  \emph{Claim:}  
    \begin{equation} \label{E:o(L).cont}
      \text{If } V \subset  P(\RR^{\nvar}) \text{ is acute, then } o(L) \text{ is continuous on } V.
    \end{equation}  
Suppose not. Then there exist $L, L_{1}, L_{2}, \ldots  \in  V$ s.t.\ 
$L_{m} \to L$, but $o(L_{m}) \to -o(L)$. Hence, for $m$ sufficiently large, 
$o(L_{m}) \cdot o(L) < 0$, contradiction. This proves the claim. 

If $L_{0} \in P(\RR^{\nvar})$ let $v_{0} \in \RR^{\nvar}$ be a unit vector in $L_{0}$. Let 
	\begin{multline} \label{E:W(L0).defn}
	  W(L_{0}) := \bigl\{ L \in P(\RR^{\nvar}) : 
	    \text{ there exists } v \in L \cap S^{\nvar -1} \text{ s.t.\ } 
	      v \cdot v_{0} > \sqrt{2}/2 \bigr\} \\
	        = \bigl\{ L \in P(\RR^{\nvar}) : 
	    \text{ there exists } v \in L \cap S^{\nvar -1} \text{ s.t.\ } 
	      \angle(v, v_{0}) < \pi/4 \bigr\} .
	\end{multline}  
Trivially, $W(L_{0})$ is independent of choice of the unit vector $v_{0} \in L_{0}$.
If $L \in W(L_{0})$, define $o(L) \in L \cap S^{\nvar -1}$ to be the unique unit vector in $L$ s.t.\ $o(L)  \cdot v_{0} > \sqrt{2}/2$. 

\emph{Claim:} 
    \begin{equation} \label{E:W(L0).acute}
      W(L_{0}) \text{ is acute, }
    \end{equation}
i.e., $L_{1}, L_{2} \in W(L_{0})$ implies $o(L_{1}) \cdot o(L_{2}) > 0$. 
Let $v_{1}, v_{2} \in \RR^{\nvar}$ be unit vectors with 
   \begin{equation}  \label{E:vi.dot.v0.bggr.thn.root.2.over.2}
      c_{i} := v_{i} \cdot v_{0} > \sqrt{2}/2
   \end{equation}
(so $0 < \sqrt{2}/2 < c_{i} \leq 1$) and let $w_{i} = c_{i} v_{0}$ 
(so $(v_{i} - w_{i}) \perp v_{0}$) and therefore $(v_{i} - w_{i}) \perp w_{1} - w_{2}$ 
($i = 1,2$). Therefore,  
   \begin{align}  \label{E:bnd.on.2.minus.2.v1.v2}
      2 - 2 v_{1} \cdot v_{2} &= |v_{1} - v_{2}|^{2} \notag \\
        &= \bigl| (v_{1} - w_{1}) + (w_{1} - w_{2}) +  (w_{2} - v_{2}) \bigr|^{2} \\
        &= |w_{1} - w_{2}|^{2} + |v_{1} - w_{1}|^{2} + |w_{2} - v_{2}|^{2} 
               + 2 (v_{1} - w_{1}) \cdot (w_{2} - v_{2}) \notag \\
        &\leq |w_{1} - w_{2}|^{2} + |v_{1} - w_{1}|^{2} + |w_{2} - v_{2}|^{2} 
          + 2 |v_{1} - w_{1}||w_{2} - v_{2}| . \notag
   \end{align}
Now
   \[
      |w_{1} - w_{2}| = |c_{1} - c_{2}| \text{ and } |v_{i} - w_{i}|^{2} = 1 - c_{i}^{2} 
        \quad (i = 1,2).
   \]
Substituting this into \eqref{E:bnd.on.2.minus.2.v1.v2} we get, after some simplification,
   \begin{equation}  \label{E:lwr.bound.on.v1.dot.v2}
      v_{1} \cdot v_{2} \geq c_{1} c_{2} - \sqrt{(1 - c_{1}^{2})(1 - c_{2}^{2})}.
   \end{equation}
Let 
   \[
      f(s,t) = s t - \sqrt{(1 - s^{2})(1 - t^{2})}, \quad (s,t \in (0,1]).
   \]
It is easy to see that $f$ is strictly increasing in its arguments. (As $s \uparrow$, then 
$st \uparrow$, but $(1 - s^{2}) \downarrow$, etc.) Therefore, by \eqref{E:vi.dot.v0.bggr.thn.root.2.over.2} and \eqref{E:lwr.bound.on.v1.dot.v2}
   \[
      v_{1} \cdot v_{2} > \left( \frac{\sqrt{2}}{2} \right)^{2} 
         - \sqrt{ \left[ 1 - \left( \frac{\sqrt{2}}{2} \right)^{2} \right]
            \left[ 1 - \left( \frac{\sqrt{2}}{2} \right)^{2} \right] } = 0.
   \]
This proves the claim \eqref{E:W(L0).acute} that $W(L_{0})$ is acute. 

Let $V \subset  P(\RR^{\nvar})$ be acute (e.g., $V = W(L_{0})$ for some $L_{0} \in P(\RR^{\nvar})$), 
let $L_{1}, \ldots , L_{m} \in  V$, and let
$\lambda_{1}, \ldots , \lambda_{m} \geq 0$ with 
$\lambda_{1} + \cdots + \lambda_{m} = 1$. First, note that 
$\sum_{i=1}^{m} \lambda_{i} o(L_{i}) \ne 0$:
   \[
      \left| \sum_{i=1}^{m} \lambda_{i} o(L_{i}) \right|^{2}
         = \sum_{i=1}^{m} \lambda_{i}^{2} \bigl| o(L_{i}) \bigr|^{2}
          + \sum_{i \ne j} \lambda_{i} \lambda_{j} o(L_{i}) \cdot o(L_{j}) 
            \geq \sum_{i=1}^{m} \lambda_{i}^{2} \ge m^{-2} > 0,
   \]
since at least one $\lambda_{i}$ exceeds $1/m$. Therefore, 
$\sum_{i=1}^{m} \lambda_{i} o(L_{i})$ spans an element of $P(\RR^{\nvar})$.

Let $\mbf{W} := \bigl\{ W(L_{0}) : L_{0} \in P(\RR^{\nvar}) \bigr\}$. We show that continuous convex combinations of finitely many lines can be defined on $\mbf{W}$. 
For $m = 1, 2, \ldots$, $V \in \mbf{W}$, $L_{1}, \ldots , L_{m} \in  V$ and $\lambda_{1}, \ldots , \lambda_{m} \geq 0$ with $\lambda_{1} + \cdots + \lambda_{m} = 1$ define 
	\begin{multline}  \label{E:regrssn.gamma.defn}
		\gamma \bigl[ V,  (\lambda_{1}, \ldots, \lambda_{m}), 
		  (L_{1}, \ldots, L_{m}) \bigr] \\
		  \text{ to be the unique line in } P(\RR^{\nvar}) \text{ containing } 
		         \lambda_{1} o(L_{1}) + \cdots + \lambda_{m} o(L_{m}).
	\end{multline}
We show that $\gamma$ (with codomain $P(\RR^{\nvar})$) is a commutative convex combination function in the sense of definition \ref{D:convex.combo.fn}. Commutativity is immediate. Trivially, if $L_{1}, \ldots, L_{m} \in W(L_{0})$ and $\lambda_{1}, \ldots, \lambda_{m} \geq 0$ with $\lambda_{1} + \cdots + \lambda_{m} = 1$ then 
$\gamma \bigl[ (\lambda_{1}, \ldots, \lambda_{m}), (L_{1}, \ldots, L_{m}) \bigr] 
\in W(L_{0})$ 
(property \ref{I:convex.combos.are.local} of definition \ref{D:convex.combo.fn}).
Continuity of $o$, \eqref{E:o(L).cont}, implies continuity of convex combination (property \ref{I:gamma.k.cont}). Properties \ref{I:consistency.of.conv.combos}, \ref{I:convex.combos.of.1.pt}, \ref{I:1.is.identity}, and \eqref{E:drop.0.coefs} of definition \ref{D:convex.combo.fn} are obvious.

Let $L \in W(L_{0})$ (see \eqref{E:W(L0).defn} and let $v, v_{0}$ be unit vectors in $L$, $L_{0}$, resp.\ s.t.\ 
$v \cdot v_{0} > \sqrt{2}/2$. Let $u := -v+2(v \cdot v_{0}) v_{0}$. Then $|u| = 1$ and so spans a line in $W(L_{0})$ but 
$v \cdot \bigl[-v+2 (v \cdot v_{0}) v_{0}  \bigr]\downarrow 0$ 
as $v \cdot v_{0} \downarrow \sqrt{2}/2$. I.e., 
	\begin{equation}  \label{E:perp.pairs.in.W.xi}
	   W(L_{0}) \text{ contains pairs of lines arbitrarily close to perpendicular.}
	\end{equation} 

Recall $S := \RR^{\nvar}$. For $L_{0} \in P(\RR^{\nvar})$, let 
	\begin{equation}  \label{E:V.xi0.defn}
		V(L_{0}) := \bigl\{ \xi \in G(\nvar-1, \nvar) : \xi \perp L 
		  \text{ for some } L \in W(L_{0}) \bigr\}.
	\end{equation}
Define a covering $\msf{V}_{90^{\circ}}$ of $G(\nvar-1,\nvar)$ by 
    \begin{equation} \label{E:V.90.defn}
      \msf{V}_{90^{\circ}} := \bigl\{ V(L_{0}) : L_{0} \in P(\RR^{\nvar}) \bigr\} .
    \end{equation} 
Since $P(\RR^{\nvar})$ and 
$G(k, \nvar)$ are homeomorphic (Milnor and Stasheff \cite[Lemma 5.1, p.\ 57]{jwMjdS74}), 
$\gamma$ corresponds to a convex combination function on $\msf{V}_{90^{\circ}}$. Hence, 
    \begin{multline}  \label{E:plane-fitting.when.nvar=k+1}
      \text{We may apply proposition \ref{P:sing.codim.in.plane-fitting} and 
        \eqref{E:Hm.a.S.lwr.bnd.for.plane.fitting.Pf} } \\
          \text{with this cover and convex combination function.}
    \end{multline}
 (Providing $\nvar = k+1$, of course).

Each set $V(L_{0}) \in \msf{V}_{90^{\circ}}$ contains planes arbitrarily close to being $90^{\circ}$ apart, in dihedral angle. For that reason we call a $\msf{V}_{90^{\circ}}$-severe singularity a ``$90^{\circ}$ singularity''. The data sets shown in the rightmost column of figure \ref{F:LS.PC.LAD.lf.plots} are all $90^{\circ}$ singularities of their respective line-fitters.  

\section{Diameter of image of neighborhoods of $90^{\circ}$ singularities} \label{SS:90.degree.sings.in.line-fitting.on.plane}
This section concerns $90^{\circ}$ singularities, i.e. $\msf{V}_{90^{\circ}}$-severe singularities, of plane-fitters with $\nvar = k+1$. (See \eqref{E:V.90.defn}.) 

To start, let $\Phi$ be a plane-fitter with $k=1$ and $\nvar=2$ (line-fitting, section \ref{SSS:LF.plots}, is a special case). Let $Y \in \Y$ be a $90^{\circ}$ singularity 
of $\Phi$ w.r.t.\ a dense set $\D' \subset \D = \Y$. Let $\clU$ be a neighborhood of $Y$ 
and let 
    \begin{equation*}
      \msf{A} := \overline{\Phi(\clU \cap \D')}.
    \end{equation*} 
Since $k = 1$, $\msf{A} \subset \F = G(1,2) = P(\RR^{2}) = P^{1}$ (see \eqref{E:proj.space.defn}). Then there exists no 
$L \in P(\RR^{2})$ s.t.\ $\msf{A} \subset V(L)$. (See \eqref{E:V.xi0.defn}.)
Examples for LS, PC, and LAD are provided by the data sets shown in the
``(c)'' panels of figure \ref{F:LS.PC.LAD.lf.plots}. In those cases $\msf{A}$ contains orthogonal pairs of lines. 

But this is not be the case in general. Consider the line fitter $\Phi_{ex}$ defined as follows. Given a data set $Y^{n \times 2}$, $\Phi_{ex}(Y)$ is defined as follows. Suppose there is a unique pair of points (i.e., rows) in $Y$ that are maximally far apart. (In particular, since $n > 2$ by \eqref{E:n>nvar>k>0}, the distance between them is strictly positive. Otherwise, the pair of points would not be unique.) Then 
$\Phi_{ex}(Y)$ is the line through the origin parallel to the unique line passing through those extremal points. 
($\Phi_{ex}(Y)$ is thus very sensitive to outliers, which makes it impractical.) If the extremal pair of points is not unique then $\Phi_{ex}(Y)$ is not defined. Suppose $n = 3$ and the rows of $Y$ are the vertices of an equilateral triangle. Then $\Phi_{ex}(Y)$ is not defined. In fact, $Y$ is a $90^{\circ}$ singularity of $\Phi_{ex}$ for which, up to rotation, 
$\msf{A}$ includes the three lines spanned by $(\cos (j \pi/3), \sin (j \pi/3))$ ($j = 0, 1, 2$). Those three lines do not lie in any $V(L)$, but the pairwise angles between them are all 
$\pi/3 < \pi/2$. Thus, the name ``$90^{\circ}$ singularity'' is a misnomer. 
 
However, we do \emph{claim:} In line-fitting on the plane, the example described in the last paragraph is the smallest $diam(\msf{A})$ (see \eqref{E:diam.of.set}) can be. 
I.e.\ in line-fitting, $\msf{A}$ corresponding to a $90^{\circ}$ singularity always contains pairs of lines at least $\pi/3$ radians apart. I.e., 
    \begin{quote}  
       (**) Let $\Phi$ is a line-fitter and let $Y$ is a $90^{\circ}$ singularity of $\Phi$ w.r.t.\ a dense set $\D' \subset \D = \Y$. Then if $\clU$ is any neighborhood of $Y$, then 
       $diam \bigl( \overline{\Phi(\clU \cap \D')} \bigr) \geq \pi/3$. 
    \end{quote} 
Thus, even though the apparent singularities of LAD near the data sets in figure \ref {F:HeightFen} are not severe, by  there is an at least codimension 2 set of data sets arbitrarily near which the LAD lines veer at least $\pi/3$ radians from each other.

Suppose (**) is false. Let $L \in \msf{A}$ be arbitrary. 
Choose a unit vector $v \perp L$. Then for every $M \in \msf{A}$ we can find a unit vector $v(M) \perp M$ 
s.t.\ $\angle \bigl( v, v(M) \bigr) < \pi/3$. 
Let $v_{+} \in S^{1}$ be the vector $v(M)$ ($M \in \msf{A}$) furthest from $v$ in the positive direction and let $v_{-}$ be similar but in the negative direction. (Possible by compactness.)
Then $v(\msf{A}) := \bigl\{ v(M) \in S^{1} : M \in \msf{A} \bigr\}$ lies between $v_{-}$ and $v_{+}$ inclusive. 
 
For $k = 1, 2, \ldots$ define:
    \begin{equation} \label{E:langle.rangle.line}
      \text{If } u \in S^{k} \text{ let } \langle u \rangle \in P^{k} 
        \text{ be the line spanned by } u.
    \end{equation} 
$\langle \cdot \rangle$ is just the restriction of the quotient map in \eqref{E:proj.space.defn} to $S^{1}$. Thus, 
    \begin{equation} \label{E:langle.rangle.is.cont}
      \langle \cdot \rangle \text{  is continuous. }
    \end{equation}
Regarding vectors in $\RR^{2}$ as row vectors, we have, by \eqref{E:rho.is.smooth}, that $\langle \cdot \rangle$ is in fact smooth.

Since $diam(\msf{A}) < \pi/3$ by supposition, we have 
$\angle \bigl( \langle v_{-} \rangle, \langle v_{+} \rangle \bigr) < \pi/3$. By \eqref{E:angle.L.M} and \eqref{E:angle.-x,y}, this means either 
$\angle(v_{-}, v_{+}) < \pi/3$ \emph{or} $\pi - \angle(v_{-}, v_{+}) < \pi/3$. 
But $\angle(v_{-}, v_{+}) \leq \angle(v_{-}, v) + \angle(v, v_{+}) < 2 \pi/3$ so 
$\pi - \angle(v_{-}, v_{+}) > \pi/3$. Therefore, $\angle(v_{-}, v_{+}) < \pi/3$. Let $w$ be the unit vector bisecting the angle (of size $< \pi/3$) between $v_{-}, v_{+}$. Thus, 
$\angle(v_{-}, w) = \angle(v_{+}, w) < \pi/6 < \pi/4$. But $v(\msf{A})$ lies between $v_{-}$ and $v_{+}$ inclusive. I.e., if $v' \in v(\msf{A})$, then $\angle(v', w) < \pi/4$.
Let $K = w^{\perp} \in P^{1}$. Then, by  \eqref{E:V.xi0.defn} and \eqref{E:W(L0).defn}, we have $\msf{A} \subset V(K)$. But this is impossible because $Y$ is a $90^{\circ}$ singularity of $\Phi$. This contradiction is a consequence of the supposition that the angle between every line in $\msf{A}$ is strictly less than $\pi/3$. This proves (**): 
$\msf{A} := \overline{\Phi(\clU \cap \D')}$ contains pairs of lines $\pi/3$ or more radians apart.

The example of $\Phi_{ex}$ with which we began shows this bound is tight. We return to this situation in remark \ref{R:90.degree.sings.in.line-fitting.again}.

Here is another example:

  \begin{example}[All singularities of least squares are $90^{\circ}$ singularities.] \label{Ex:all.LS.sings.are.90.degrees}
Consider least squares linear regression (LS, section \ref{SS:lin.reg.and.LS}) with 
$m = 1$ in \eqref{E:n.>.k+m}. Then $k = \nvar - 1$. We show that 
every singularity of LS is a $90^{\circ}$ singularity. ($m=1$ necessary in order that $90^{\circ}$ singularity makes sense.) For let $\Phi_{LS}$ be the LS plane-fitter and let $Y \in \Y$ be a singularity of $\Phi_{LS}$. Let $\Y' \subset \Y$ be the set of non-collinear data sets, let $\clU \subset \Y$ be a neighborhood of $Y$, 
and let $\msf{A} :=\overline{\Phi_{LS}(\clU \cap \Y'})$. Then by proposition \ref{P:collin.data.are.sings.of.LS} and remark \ref{R:F.can.be.orthog.to.xi}, for some 
$k' = 0, 1, \ldots, k-1$ and $\xi \in G(k',\nvar)$, we have that $\msf{A}$ contains an isometric image of $\{\xi\} \times G(k-k', k-k'+1)  \subset G(k, \nvar)$. Thus, $\msf{A}$ contains planes 
$\omega, \omega'$ s.t.\ $d(\omega, \omega') = \pi/2$. (See \eqref{E:metric.on.G(k,k+1)}.) Hence, $\msf{A}$ lies in no $V \subset \msf{V}_{90^{\circ}}$. Hence, $Y$ is a $90^{\circ}$ singularity. It also follows that the diameter, $diam(\msf{A})$, of $\msf{A}$ is $\pi/2$.

To make a short story long, let $Y = (X,y) \in \Y$ be a singularity of LS w.r.t.\ $\Y'$. Then, by proposition \ref{P:collin.data.are.sings.of.LS} (with $m=1$), $Y$ is collinear. Recall the definition, \eqref{E:X1.defn}, of $X_{1}$. In proposition \ref{P:collin.data.are.sings.of.LS}, $k' := rank \, X_{1} - 1$. Therefore, by lemma \ref{L:Y.collnr.iff.rank.X1.<.k+1}, 
$k' \leq k-1$. Let $\ell := k-k'+1 \geq 2$. Let 
    \begin{equation*}
       \xi \in G(k', k+1) \text{ and } F : \RR^{\ell} \to \RR^{k +1} = \RR^{\nvar}
    \end{equation*}
correspond to $Y$ as described in proposition \ref{P:collin.data.are.sings.of.LS}. Let 
    \begin{equation*}
      \omega := F(\RR^{\ell}) \in G(\ell, k +1) .
    \end{equation*}
Then by proposition \ref{P:collin.data.are.sings.of.LS} yet again, 
$\omega \cap \xi = \{0\}$. Let $A^{\ell \times \nvar}$ (recall $\nvar = k+1$ here) satisfy $\rho(A) = \omega$, where as usual ``$\rho(\cdot)$'' denotes the row space functional. $A$ has full rank $\ell$. 

Let $\Pi_{\xi}^{\nvar \times \nvar}$ be the matrix of orthogonal projection 
of $\RR^{k +1}$ onto $\xi$. (Again, $\nvar = k+1$.)
Define $B^{\ell \times \nvar} := A(I_{\nvar} - \Pi_{\xi})$. ($I_{\nvar}$ is the $\nvar$-dimensional identity matrix.) If $rank \, B < \ell$, then there exists $a^{\ell \times 1} \neq 0$ s.t.\ 
$0 = a B = a A - a A \Pi_{\xi}$. Thus, $a A$ is a nonzero point ($A$ has full rank, 
$a \neq 0$) in $\omega$ that is equal to the point $a A \Pi_{\xi} \in \xi$. That contradicts 
$\omega \cap \xi = \{0\}$. Therefore, the eigenvalues of $(B B^{T})^{\ell \times \ell}$ are all strictly positive.

Let $u_{i}^{1 \times \ell}$ ($i=1,2$) be orthonormal eigenvectors 
of $B B^{T}$. ($B B^{T}$ is $\ell \times \ell$. $\ell \geq 2$.) 
Thus, $u_{i} B \neq 0$ ($i=1,2$). For $i=1,2$, 
let $\alpha_{i} := \{ b A \in \omega :  b^{1 \times \ell} \perp u_{i} \}$. (See \eqref{E:superscript.perp.notation}.) Because $A$ is of full rank $\ell$, $\alpha_{i}$ 
is an $(\ell -1 = k-k')$-dimensional subspace of $\omega$. 
Since $\alpha_{i} \subset \omega$ we have  $\xi \cap \alpha_{i} = \{0\}$. 
Let $\zeta_{i} := F^{-1}(\alpha_{i})$. Hence, $\zeta_{i}$ is an $(\ell -1 = k-k')$-dimensional subspace of $\RR^{\ell}$ and  
$\xi + F(\zeta_{i}) = \xi + \alpha_{i} \in G(k, k+1)$. If $y \in \zeta_{i}$, 
then for some $b_{i}^{1 \times \ell}$, 
we have $F(y) = b_{i} A \in \omega \subset \RR^{k +1}$, where $b_{i} \perp u_{i}$. 

\emph{Claim:} 
    \begin{equation} \label{E:u.B.perp.xi+alpha}
      u_{i} B \perp \xi + F(\zeta_{i}) = \xi + \alpha_{i}, \qquad (i = 1,2). 
    \end{equation}
Let 
    \begin{equation*}
      z^{1 \times (k+1)} \in \xi + \alpha_{i} \in G(k, k+1) . 
    \end{equation*}
By \eqref{E:proj.mat.symmetric.idempotent},
    \begin{multline}  \label{E:(z-zPi.dot.u.=z.dot.uB}
      (z - z \Pi_{\xi}) \cdot u_{i} B = z(I_{\nvar} - \Pi_{\xi}) B^{T} u_{i}^{T}
        = z(I_{\nvar} - \Pi_{\xi}) (I_{\nvar} - \Pi_{\xi}) A^{T} u_{i}^{T} \\
          = z(I_{\nvar} - \Pi_{\xi}) A^{T} u_{i}^{T} = z B^{T} u_{i}^{T}  = z \cdot u_{i} B .
    \end{multline}
Thus, $u_{i} B \perp z$ if and only if $u_{i} B \perp (z - z \Pi_{\xi})$. 
Write $z = x + a$, where $x \in \xi$ and $a \in \alpha_{i}$. Notice, 
$z - z \Pi_{\xi}= a - a \Pi_{\xi}$. 
Since $a \in \alpha_{i}$, we can write $a = b_{i} A$, where $b_{i} \perp u_{i}$. Hence, 
    \begin{equation*}
      z - z \Pi_{\xi} = a - a \Pi_{\xi} = b_{i} A - b_{i} A \Pi_{\xi} 
        = b_{i} \bigl[ A (I_{\nvar} - \Pi_{\xi}) \bigr] = b_{i} B .
    \end{equation*} 
We thus have $u_{i} B \cdot (z - z \Pi_{\xi}) = u_{i} B B^{T} b_{i}^{T}$. 
But $b_{i} \perp u_{i}$ and $u_{i}$ is an eigenvector of $B B^{T}$. 
Therefore, by \eqref{E:(z-zPi.dot.u.=z.dot.uB}, 
$u_{i} B \cdot z = u_{i} B \cdot (z - z \Pi_{\xi}) \propto u_{i} b_{i}^{T} = 0$. But $z$ is an arbitrary element of $\xi + \alpha_{i}$. Thus, $u_{i} B \perp \xi + F(\zeta_{i})$ 
($i=1,2$), proving the claim, \eqref{E:u.B.perp.xi+alpha}.

Now, $u_{1}$ and $u_{2}$ are orthonormal eigenvectors of $BB^{T}$. 
Hence, $u_{1} B \cdot u_{2} B = u_{1} B B^{T} u_{2}^{T} = 0$. I.e., $u_{1} B \perp u_{2} B$. At the same time, $u_{i} B \perp \xi + F(\zeta_{i}) \in G(k, k+1)$ ($i=1,2$). Therefore, by \eqref{E:u.B.perp.xi+alpha}, \eqref{E:V.xi0.defn}, and \eqref{E:W(L0).acute},
we have that $\xi + F(\zeta_{1})$, and $\xi + F(\zeta_{2})$ cannot both belong to the same 
$V \in \msf{V}_{90^{\circ}}$. But, by proposition \ref{P:collin.data.are.sings.of.LS}, for $i=1,2$ there is a sequence $\{ Y_{i,m} \} \subset \Y'$ s.t.\ $Y_{i,m} \to Y$
and $\Phi_{LS}(Y_{i,m}) \to \xi + F(\zeta_{i})$ as $m \to \infty$ 
so $\xi + F(\zeta_{i}) \in \msf{A}$ ($i = 1,2$). 
Thus, $Y$ is a $90^{\circ}$ singularity of LS and $diam(\msf{A}) = \pi/2$. But $Y \in (\Y')^{c}$ is an arbitrary singularity of LS, that is to say an arbitrary collinear data set.
  \end{example}

Let $\Phi$ be a plane-fitter 
(with $\nvar = k+1$) and let $\D'$ be the dense subset of $\D$ on which $\Phi$ is defined and continuous and relative to which singularities of $\Phi$ are defined, let $x$ be a $90^{\circ}$ singularity of $\Phi$, let $\clU$ is a neighborhood of $x$, and let $V \in \msf{V}_{90^{\circ}}$. 
Define $\msf{A} := \overline{\Phi(\clU \cap \D')}$. 
$\msf{A}$ is a compact subset of $G(k, \nvar)$ and $\msf{A} \nsubseteqq V$. By \eqref{E:perp.pairs.in.W.xi} and \eqref{E:V.xi0.defn}, $V$ contains pairs of planes arbitrarily close to being $\pi/2$ apart. (See \eqref{E:metric.on.G(k,k+1)}.) But as we have seen in section \ref{SS:90.degree.sings.in.line-fitting.on.plane}, it does not follow that $diam(\msf{A}) \geq \pi/2$. 

However, there is something that can be said in general about the size 
of $\msf{A}$ \emph{providing we orient all the planes in} $\msf{A}$. Let 
$P^{k} = P(\RR^{\nvar})$ be real $k$-dimensional projective space. 
Let $v \in S^{k}$ and let 
    \begin{equation} \label{E:X(v,pi/2).defn}
      X(v, \pi/2) := \{ x \in S^{k} : x \cdot v \geq 0 \} 
    \end{equation}
so $X(v, \pi/2)$ is a closed hemisphere. By Milnor and Stasheff \cite[lemma 5.1 and remark following it, p.\ 57]{jwMjdS74}, the map that takes $\xi \in G(k, \nvar)$ 
to $\xi^{\perp} \in P^{k}$ is continuous. Let $\msf{A}^{\perp} := \{ L \in P^{k} : \text{ there exists } \xi \in \msf{A} \text{ s.t.\ } L \perp \xi \}$. 
Thus, $\msf{A}^{\perp}$ is compact.

Recall the definition \eqref{E:langle.rangle.line} of $\langle \cdot \rangle : S^{k} \to P^{k}$. If $L \in P^{k}$ then $\langle \rangle^{-1}(L)$ consists of the two vectors in $S^{k}$ that each span $L$. Define the obvious notation:  
$\langle \rangle^{-1}(\msf{A}^{\perp}) := \bigcup_{L \in \msf{A}^{\perp}} 
\langle \rangle^{-1}(L)$. Let $v \in S^{k}$. The maximum angle between any two lines in $P^{k}$ is $\pi/2$. Thus, for every $L \in \msf{A}^{\perp}$ there exists 
$w \in X(v, \pi/2)$ s.t.\ $w \in L$, i.e., $w$ spans $L$. $o(\msf{A},v)$ is the set of all such:
    \begin{equation}  \label{E:o(A,v).defn}
      o(\msf{A},v) := \langle \rangle^{-1}(\msf{A}^{\perp}) \cap X(v, \pi/2) 
        = \{ w \in X(v, \pi/2) : 
          \text{there exists } \xi \in \msf{A} \text{ s.t.\ } w \perp \xi \} .
    \end{equation}
Since $\langle \cdot \rangle$ is continuous, by \eqref{E:langle.rangle.is.cont}, and
$X(v, \pi/2)$ and $\msf{A}^{\perp}$ are compact we have that 
    \begin{equation}  \label{E:o(A,v).is.compact}
      o(\msf{A},v) \text{ is compact.}
    \end{equation}

If $\xi \in \msf{A}$, $v \in S^{k}$, and $v \notin \xi$ 
(so $\xi^{\perp}$ is not perpendicular to $v$) then $\xi$ is represented by exactly one vector in $o(\msf{A},v)$. (I.e., there exists exactly one $w \in o(\msf{A},v)$ s.t.\ $w \perp \xi$.) But if $v \in \xi$ (so $\xi^{\perp}$ \emph{is} perpendicular to $v$) then $\xi$ is represented by \emph{two} antipodal vectors in $o(\msf{A},v)$. 

One can ask what is the smallest diameter of $o(\msf{A},v)$ possible, for any $v$? Another way to put it is, suppose instead of measuring the distance between planes $\xi, \zeta$ by \eqref{E:metric.on.G(k,k+1)} (and \eqref{E:angle.L.M}), we start with an arbitrary $v_{0} \in S^{k}$ and measure the distance between the planes as follows. As just observed, there might be two vectors $v \in S^{k}$ s.t.\ $v \perp \xi$ and 
$v \cdot v_{0} \geq 0$. But suppose there is exactly one $v \in S^{k}$ s.t.\ 
$v \perp \xi$ and $v \cdot v_{0} \geq 0$. Suppose there is exactly one $w \in S^{k}$ s.t.\ $w \perp \zeta$ and $w \cdot v_{0} \geq 0$. Define the distance between $\xi$ and 
$\zeta$ to be $\angle(v,w)$. This can be thought of as the distance between the planes oriented by $v_{0}$. Relative to that distance the diameter (see \eqref{E:diam.of.set}) 
of $\msf{A} \subset G(k, \nvar)$ is the same as the diameter 
of $o(\msf{A},v) \subset S^{k}$ and has to be at least $\pi/3$. In fact:

   \begin{prop} \label{P:diam.oriented.90.degree.sing}
Let $\Phi$ be a plane-fitter (with $\nvar = k+1$) and let $\D'$ be the dense subset of $\D$ on which $\Phi$ is defined and continuous and relative to which singularities of $\Phi$ are defined. 
Let $Y \in \Ss^{\msf{V}_{90^{\circ}}}$. Let $\clU$ be a neighborhood of $Y$. 
Let $\msf{A} := \overline{\Phi(\clU \cap \D')} \subset G(k,\nvar)$.  

Let $v_{0} \in S^{k}$ be arbitrary. Then for every $\epsilon > 0$ we have the following. Let $i=1,2$. There exists $Y_{i} \in \clU \cap \D'$ and 
$w_{i} \in S^{k}$ with $w_{i} \perp \Phi(Y_{i})$ and $w_{i} \cdot v_{0} \geq 0$ (so $w_{i} \in o(\msf{A},v_{0})$) s.t.\ 
    \begin{equation} \label{E:oriented.angle.lwr.bnd}
      \angle (w_{1}, w_{2}) 
        > 2 \arcsin \bigl( (1/2) \sqrt{  (k + 1)/k } \bigr) - \epsilon .
    \end{equation}
Then 
    \begin{equation*}
      diam \bigl[ o(\msf{A},v_{0}) \bigr] 
        \geq 2 \arcsin \bigl( (1/2) \sqrt{  (k + 1)/k } \bigr) .
    \end{equation*}
In particular, $diam \bigl[ o(\msf{A},v_{0}) \bigr] > \pi/3$.
   \end{prop}
Suppose there exists $\xi \in \msf{A}$ s.t.\  $v_{0} \in \xi$ and let $w \in S^{k}$ be orthogonal to $\xi$. Then $\{w, -w\} \subset o(\msf{A},v_{0})$ so 
$diam \bigl[ o(\msf{A},v_{0}) \bigr] = \pi$, a not very interesting case. 
  \begin{proof} Suppose $Y \in \Ss^{\msf{V}_{90^{\circ}}}$ (see \eqref{E:V.90.defn}) and let $\clU$ be a neighborhood of $Y$. Let $L_{0} \in P^{k} = P(\RR^{\nvar})$, 
$k$-dimensional projective space. Then the closure, 
    \begin{equation*}
      \msf{A} := \overline{\Phi(\clU \cap \D')},
    \end{equation*}
  of the image of $\clU \cap \D'$ does not lie 
  in $V(L_{0})$ (see \eqref{E:V.xi0.defn}). 
  
If $v \in S^{k}$ and $\theta \in [0, \pi]$ let 
     \begin{equation}  \label{E:X(v.theta).defn}
      X(v, \theta) := \{ x \in S^{k} : x \cdot v \geq \cos \theta \}
        = \{ x \in S^{k} : 0 \leq \angle(x, v) \leq \theta \} .
    \end{equation} 
($\cos$ is decreasing on $[0, \pi]$.) Let $v_{0} \in S^{k}$ be arbitrary. Let 
    \begin{equation}  \label{E:theta.v0.defn}
      \theta_{v_{0}} := 
          \inf \bigl\{ \theta \in [0, \pi] : 
        \text{There exists } v \in S^{k} \text{ s.t.\ } o(\msf{A},v_{0}) 
          \subset X(v, \theta) \bigr\} .
    \end{equation}
(Note that it is $o(\msf{A},v_{0})$ that must lie in $X(v, \theta)$, not $o(\msf{A}, v)$.) Since by \eqref{E:o(A,v).defn}, $o(\msf{A},v_{0}) \subset X(v_{0}, \pi/2)$, 
we have $\theta_{v_{0}} \leq \pi/2$.

\emph{Claim:}  
    \begin{equation}  \label{E:theta.v0.geq.pi/4}
       \pi/2 \geq \theta_{v_{0}} \geq \pi/4 .
    \end{equation} 
We have just agreed that $\pi/2 \geq \theta_{v_{0}}$. Suppose 
$\pi/4 > \theta_{v_{0}}$. Then there exists $\theta \in [0,\pi/4)$ and 
$v \in S^{k}$ s.t.\ 
    \begin{equation}  \label{E:o(A,v0).subset.X(v.theta}
      o(\msf{A}, v_{0}) \subset X(v,\theta) .
    \end{equation}
Let $L_{v}$ be the span, 
$\langle v \rangle \in P(\RR^{\nvar})$, of $v$. (See \eqref{E:langle.rangle.line}.) Since 
$Y \in \Ss^{\msf{V}_{90^{\circ}}}$ (see \eqref{E:V.90.defn}), we have 
$\msf{A} \nsubseteq V(L_{v})$. (See \eqref{E:V.xi0.defn}.) Let $\xi \in \msf{A} \setminus V(L_{v})$. There exists $w \in S^{k}$ s.t.\ $w \perp \xi$ and $w \cdot v_{0} \geq 0$. By \eqref{E:X(v,pi/2).defn} and \eqref{E:o(A,v).defn}, 
$w \in o(\msf{A}, v_{0})$ so, by \eqref{E:o(A,v0).subset.X(v.theta}, $w \in X(v, \theta)$. Therefore, by \eqref{E:X(v.theta).defn}, 
$\angle(w,v) \leq \theta < \pi/4$. Hence, $L := \langle w \rangle \in W(L_{v})$. (See \eqref{E:W(L0).defn}.) Therefore $\xi \in V(L_{v})$. Contradiction.

We will use the following.
  \begin{lemma} \label{L:max.dist.is.cont}
Let $Z$ be a metric space with finite metric $\mu$ and let $K \subset Z$ be compact. Then the function $z \mapsto \max \{\mu(z, x) : x \in K \} \in [0, \infty)$ is continuous on $Z$.
  \end{lemma}
  \begin{proof}[Proof of lemma]
Note that the map that takes $z \in Z$ to the compact set $\{z\} \subset Z$ is continuous (in fact an isometric imbedding) w.r.t.\ $ \mu $ and the Hausdorff distance (Tuzhilin \cite{aaT2019.HausdorffGromovHausdorffDistance}; see also the Wikipedia article about Hausdorff distance) on the space of compact subsets of $Z$. The space of compact subsets of $Z$ is compact w.r.t.\ the Hausdorff distance. Moreover, if $z \in Z$, then 
$\max \{ \mu (z, x) : x \in K \}$ is just the Hausdorff distance from $\{z\}$ to $K$. Since a metric on a space is  continuous w.r.t.\ itself, the lemma follows. 
  \end{proof}

\emph{Proof of proposition \ref{P:diam.oriented.90.degree.sing} continued.} 
For $v \in S^{k}$, let 
$F(v) := F_{v_{0}}(v) := \max \bigl\{ \angle(w, v) : w \in o(\msf{A},v_{0}) \bigr\}$. 
$\angle$ is a metric on $S^{k}$. Since $o(\msf{A},v_{0})$ is compact, by lemma \ref{L:max.dist.is.cont}, $F$ is continuous on $o(\msf{A},v_{0})$. Let $v \in S^{k}$. By \eqref{E:theta.v0.defn}, a necessary condition that $o(\msf{A},v_{0}) \subset X(v, \theta)$ is 
$\theta \geq \theta_{v_{0}}$. By \eqref{E:X(v.theta).defn},
        \begin{align}  \label{E:o.X.F}
           o(\msf{A},v_{0}) \subset X(v, \theta)
             &\text{ if and only if for all } w \in o(\msf{A},v_{0}) \text{ we have }
               \angle(w,v) \leq \theta,  \\
             &\text{ if and only if } F(v) \leq \theta \notag .
        \end{align} 
Since $F$ is continuous, it achieves a minimum value on the compact set 
$o(\msf{A},v_{0})$. Let $v_{\msf{A}} \in S^{k}$ be a point where $F$ achieves its minimum value and let $\theta_{min}$ be that minimum value. 
We must have $\theta_{min} \leq \pi/2$, otherwise, by \eqref{E:o(A,v).defn}, we would have $F(v_{0}) < F(v_{\msf{A}})$, contradicting the definition of $v_{\msf{A}}$.
By \eqref{E:o.X.F}, $o(\msf{A},v_{0}) \subset X(v_{\msf{A}}, \theta_{min})$. Hence, by \eqref{E:theta.v0.defn}, $\theta_{min} \geq \theta_{v_{0}}$. But by definition 
of $\theta_{min}$, $\theta_{min} \leq  \theta_{v_{0}}$. I.e., $\theta_{min} = \theta_{v_{0}}$. To emphasize the connection of 
$\theta_{min} = \theta_{v_{0}}$ to $v_{\msf{A}}$ we define 
    \begin{equation} \label{E:three.thetas}
      \theta_{\msf{A}} := \theta_{min} = \theta_{v_{0}} .
    \end{equation}
We conclude
    \begin{equation}  \label{E:o(A,v0).subset.X(vA,theta.v0)}
       o(\msf{A},v_{0}) \subset X(v_{\msf{A}}, \theta_{\msf{A}}) 
         \text{ so } x \cdot v_{\msf{A}} \geq \cos \theta_{\msf{A}} 
           \text{ for every } x \in o(\msf{A},v_{0}).
    \end{equation}

Let 
    \begin{equation}  \label{E:Bd.o(A).defn}
      Bd_{v_{\msf{A}}} \, o(\msf{A},v_{0}) := 
        \{ x \in o(\msf{A},v_{0}) : x \cdot v_{\msf{A}} = \cos \theta_{\msf{A}} \}.
    \end{equation} 
Since $F(v_{\msf{A}}) = \theta_{\msf{A}}$, $Bd_{v_{\msf{A}}}$ is nonempty. 
($Bd_{v_{\msf{A}}} \, o(\msf{A},v_{0})$ is not literally the topological boundary of $o(\msf{A},v_{0})$.)  

Thus, $Bd_{v_{\msf{A}}} \, o(\msf{A},v_{0})$ lies in the $k$-plane 
    \begin{equation*}
      P := \{ x \in \RR^{\nvar} : x \cdot v_{\msf{A}} = \cos \theta_{\msf{A}} \} .
    \end{equation*} 
I.e.,
    \begin{equation} \label{E:Bdo.subset.P}
      Bd_{v_{\msf{A}}} \, o(\msf{A},v_{0}) = o(\msf{A},v_{0}) \cap P .
    \end{equation}
Hence, $Bd_{v_{\msf{A}}} \, o(\msf{A},v_{0})$ is compact.

Let 
    \begin{equation*}
      c := (\cos \theta_{\msf{A}}) \, v_{\msf{A}} \in P \subset \RR^{\nvar} .
    \end{equation*}
$c = 0 \in \RR^{\nvar}$ is possible. Thus, 
    \begin{equation} \label{E:P=vA.perp+c}
       P = v_{\msf{A}}^{\perp} + c.
    \end{equation} 
Let $x \in P \cap S^{k}$. E.g., $x$ might be in $o(\msf{A},v_{0})$. Write $x = z + c$, where $z \perp v_{\msf{A}}$. Since $o(\msf{A},v_{\msf{A}} ) \subset S^{k}$, 
$1 = |z|^{2} + |c|^{2}$, 
so $|z|^{2} = 1 - |c|^{2} = 1 - \cos^{2} \theta_{\msf{A}} = \sin^{2} \theta_{\msf{A}}$. On the other hand, $|x - c|^{2} = |z|^{2} = \sin^{2} \theta_{\msf{A}}$. Thus, 
    \begin{multline} \label{E:Bd.o(A).S.sin.theta.A}
      Bd_{v_{\msf{A}}} \, o(\msf{A},v_{0}) \text{ lies on the sphere } S := P \cap S^{k}. 
        \text{ The center of } S \text{ is } c \in P \\
          \text{ and the radius is } s := \sin \theta_{\msf{A}} . 
    \end{multline}
We have  
    \begin{equation} \label{E:s.geq.sqrt(2)/2}
      s = \sin \theta_{\msf{A}} \geq \sqrt{2}/2 ,
    \end{equation}
because if $s < \sqrt{2}/2$ then $\cos \theta_{\msf{A}} = \sqrt{1 - s^{2}} >  \sqrt{2}/2$, contradicting \eqref{E:theta.v0.geq.pi/4} and \eqref{E:three.thetas}. 
See appendix \ref{Chptr:misc.proofs} for the proof of the following.
  \begin{lemma} \label{L:Bd.va.o(A,v0).does.not.lie}
$Bd_{v_{\msf{A}}} \, o(\msf{A},v_{0})$ does not lie in a ball in $P$ of radius smaller than $s = \sin \theta_{\msf{A}}$. 
  \end{lemma}

The Euclidean diameter of $o(\msf{A},v_{0})$ is at least 
$\delta := diam \bigl(  Bd_{v_{\msf{A}}} \, o(\msf{A},v_{0}) \bigr)$, the diameter 
of $Bd_{v_{\msf{A}}} \, o(\msf{A},v_{0})$  measured in Euclidean distance in $P$, which is the same as the diameter measured in Euclidean distance in 
$\RR^{\nvar} = \RR^{k+1}$. We have just seen that the smallest ball in $P$ containing 
$Bd_{v_{\msf{A}}} \, o(\msf{A},v_{0})$ is $S$, which has radius $\sin \theta_{\msf{A}}$. 
(See \eqref{E:Bd.o(A).S.sin.theta.A}.) By \eqref{E:s.geq.sqrt(2)/2} and Jung's theorem (Jung \cite{hJ1901.Jungs.theorem}, \cite{hJ1910.Jungs.theorem}, 
Federer \cite[2.10.41, p.\ 200]{hF69}, Rademacher and Toeplitz \cite[Chapter 16]{hRoT1957.EnjoymentOfMath}, Wikipedia), 
    \begin{equation*}
      \delta \geq \sin \theta_{\msf{A}} \sqrt{ \frac{2 (\dim P + 1)}{\dim P} }
        \geq \frac{\sqrt{2}}{2} \sqrt{ \frac{2(k + 1)}{k} } = \sqrt{ \frac{k + 1}{k} } 
          > 1 .
    \end{equation*}

The length of arc in $S^{k}$ connecting two points in $S^{k}$ that are at least
$\sqrt{ (k + 1)/k }$ units apart in Euclidean distance is at least
    \begin{equation} \label{E:min.diam.image.90deg.sing}
      d_{k} := 2 \arcsin \bigl( (1/2) \sqrt{  (k + 1)/k } \bigr) \in c(\pi/3, \pi/2] ,
    \end{equation}
by \eqref{E:n>nvar>k>0}. Thus, since $o(\msf{A},v)$ is compact, there exist 
$z_{1}, z_{2} \in o(\msf{A},v_{0})$ s.t.\ 
$\angle(z_{1}, z_{2}) \geq$ \linebreak
$2 \arcsin \bigl( (1/2) \sqrt{  (k + 1)/k } \bigr)$.  
By definition of $\msf{A}$, there exist $Y_{1}, Y_{2} \in \clU \cap \D'$ s.t.\ $w_{i} \perp \Phi(Y_{i})$ with $w_{i} \cdot v_{0} \geq 0$ and $w_{i}$ is arbitrarily close to $z_{i}$ ($i=1,2)$. But $v_{0}$ is an arbitrary point in $S^{k}$. The proposition follows.
  \end{proof}

  \begin{remark} 
\label{R:90.degree.sings.in.line-fitting.again}
 If we restrict ourselves to oriented lines, in line-fitting ($k = 1$, 
$\nvar = 2$), the proposition tells us that the diameter of $o(\msf{A},v)$ is at least 
$\pi/2$. For example, in the case of the the line fitter $\Phi_{ex}$ defined in section \ref{SS:90.degree.sings.in.line-fitting.on.plane}, if $n = 3$ and the rows of $Y$ are the vertices of an equilateral triangle. We have 
$diam \bigl[ o(\msf{A},v) \bigr] \geq 2 \pi/3$, which is bigger than $\pi/2$, the lower bound guaranteed by the proposition (when $k=1$). However, as we saw in section \ref{SS:90.degree.sings.in.line-fitting.on.plane}, in the case of $\Phi_{ex}$ the diameter of $\msf{A}$ unoriented (the same as $diam(\msf{A})$) is $\pi/3$.
  \end{remark}

To handle the unoriented case a version of Jung's theorem in real projective space (Dekster \cite{bvD97.JungThmOnManifs}) might be useful. 

\subsection{Stability of regression coefficients near $90^{\circ}$ singularities} \label{SSS:reg.coefs,near.90.degree.sings}
In linear regression the interest is not in planes but in the coefficients of the affine functions whose graphs the planes are. (See subsection \ref{SS:functions.vs.geometry}.) Using the result of proposition \ref{P:diam.oriented.90.degree.sing}, we show that $90^{\circ}$ singularities (see \eqref{E:V.90.defn}) have a serious impact on regression coefficients, at least if $m = 1$ (look near \eqref{E:n.>.k+m}). 

Let $\Phi$ be the plane fitter corresponding to a linear regression method, $R$. Suppose $m$, the dimension of the response variable is $m=1$. Let $\D'$ be the dense subset 
of $\Y$, \eqref{E:mcl.Y,defn}, w.r.t.\ which singularity is defined in the setting of interest.

If $Y \in \D'$ let 
$\bigl\{ \bigl( x, \, a(Y)^{1 \times 1} + x^{1 \times k} b(Y)^{k \times 1} \bigr) : x \in \RR^{k}) \bigr\}$ be the $k$-plane fitted to $Y$ by $R$. \emph{We assume that} $|b(Y)| < \infty$ for every $Y \in \D'$. By \eqref{E:Phi(Y).orthog.to.in.span.of}, 
    \begin{equation}  \label{E:(b(Y),-1).perp.Phi(Y)}
      \bigl( b(Y)^{T}, -1 \bigr)^{1 \times \nvar} \perp \Phi(Y) .
    \end{equation}

Let $Y_{0} \in \Y$ be a $90^{\circ}$ singularity of $\Phi$ and let $\clU \subset \Y$ be a neighborhood of $Y_{0}$. We analyze the diameter of 
$b(\clU \cap \D') = \bigl\{ b(Y) \in \RR^{k} ; Y \in \clU \cap \D' \bigr\}$. 
If $b$ is unbounded on $\clU \cap \D'$ then 
$diam \bigl( b(\clU \cap \D') \bigr) = +\infty$. (See \eqref{E:diam.of.set}.) So suppose
    \begin{equation}   \label{E:reg.coef.bnd}
      b \text{ is bounded on } \clU \cap \D' .
    \end{equation}
Let
    \begin{equation*}
      \msf{A} :=\overline{\Phi(\clU \cap \D')} \subset G(k, k+1) .
    \end{equation*}
Let 
    \begin{equation*}
      v_{0} = (0^{1 \times k}, -1) \in S^{k} \subset \RR^{\nvar} .
    \end{equation*}

Define 
    \begin{equation}  \label{E:o[Y].defn}
      o[Y] := \bigl( |b(Y)|^{2} + 1 \bigr)^{-1/2} \bigl( b(Y)^{T}, \, -1 \bigr) \in S^{k-1} ,  
            \qquad Y \in \D' ,
    \end{equation}
Then, by \eqref{E:(b(Y),-1).perp.Phi(Y)}, 
    \begin{equation}  \label{E:o[Y].perp.Phi(Y)}
      \text{ for every } Y \in \clU \text{ we have } o(Y) \perp \Phi(Y) .
    \end{equation}
Then, by \eqref{E:reg.coef.bnd}, 
    \begin{equation}  \label{E:o[Y].dot.v0.bnd.above.0} 
      \text{there exists } \epsilon > 0 \text{ s.t.\ } o[Y] \cdot v_{0} > \epsilon .
        \text{ for every } Y \in \clU .
    \end{equation}
Let
    \begin{equation*}
      B := \bigl\{ o[Y] \in S^{k-1} : Y \in \clU \cap \D' \bigr\} .
    \end{equation*}
Thus, $B = o[\clU \cap \D']$.  If $Y \in \D'$ then $|b(Y)| < \infty$ so we can recover $b(Y)$ from $o[Y]$. Thus, $B$ is directly relevant to the behavior of the regression method $R$ near $Y_{0}$..

Define $o(\msf{A},v_{0})$ as in \eqref{E:o(A,v).defn}.  By \eqref{E:o(A,v).is.compact}, $o(\msf{A},v_{0})$ is compact, hence closed. \emph{Claim:} $o(\msf{A},v_{0})$ is the closure, $\overline{B}$, of $B$. To see this, first 
let $w \in B$. Then for some $Y \in \clU \cap \D'$ we have 
$w = o[Y]$, so $w \in S^{k}$ (since, by \eqref{E:reg.coef.bnd}, $b(Y)$ is finite 
in $Y \in \D'$). Moreover, $w \cdot v_{0} > 0$, by \eqref{E:o[Y].dot.v0.bnd.above.0}, and 
$w \perp \Phi(Y)$, by \eqref{E:o[Y].perp.Phi(Y)}. 
Hence, $w \in o(\msf{A},v_{0})$. Since $o(\msf{A},v_{0})$ is closed, this shows 
$\overline{B} \subset o(\msf{A},v_{0})$. 

Conversely, let $w \in o(\msf{A},v_{0})$. Then $w \cdot v_{0} \geq 0$ and there exists 
$\xi \in \msf{A}$ s.t.\ $w \perp \xi$. Arbitrarily close to $\xi$ is a plane 
$\zeta \in \Phi(\clU \cap \D')$. Say, $\zeta = \Phi(Y)$, where $Y \in \clU \cap \D'$. 
Write $w(\zeta) := o[Y]$. By \eqref{E:o[Y].perp.Phi(Y)} and \eqref{E:o[Y].dot.v0.bnd.above.0},  
    \begin{equation}  \label{E:w(zeta).dot.v0.>eps}
      w(\zeta) \perp \zeta \text{ and } w(\zeta) \cdot v_{0} > \epsilon. 
    \end{equation}  
Let $\{\zeta_{m}\} \subset \Phi(\clU \cap \D')$ be an arbitrary sequence 
in $\Phi(\clU \cap \D')$ converging to $\xi$. There exists $Y_{m} \in \clU \cap \D'$ s.t.\ 
$\zeta_{m} = \Phi(\Y_{m})$. By \eqref{E:w(zeta).dot.v0.>eps}, $w(\zeta_{m}) = o(Y_{m}) \perp \zeta_{m}$ and $w(\zeta_{m}) \cdot v_{0} > \epsilon > 0$. 

Recall \eqref{E:langle.rangle.line} and \eqref{E:metric.on.G(k,k+1)}. 
By \eqref{E:angle.L.M}, for every $m \in \NN$ there exists $o(\xi)_{m} \in S^{k-1}$ s.t.\ $o(\xi)_{m} \perp \xi$ and 
$d(\zeta_{m}, \xi) = \angle \bigl( w(\zeta_{m}) , o(\xi)_{m} \bigr)$. Thus, 
$\angle \bigl( w(\zeta_{m}) , o(\xi)_{m} \bigr) \to 0$ as $m \to \infty$. 
We must have $o(\xi)_{m} = \pm w$. Suppose there exists a subsequence 
$\{w(\zeta_{m_{j}})\}$ converging to $-w$. Then 
$0 < \epsilon \leq w(\zeta_{m_{j}}) \cdot v_{0} \to - w \cdot v_{0} \leq 0$. Contradiction. Hence, $o(Y_{m}) = w(\zeta_{m_{j}}) \to w$. I.e., $w \in \overline{B}$. This proves the claim that $o(\msf{A},v_{0}) = \overline{B}$.
 
Thus we have, $diam(B) = diam \bigl[ o(\msf{A},v_{0}) \bigr]$, so by proposition \ref{P:diam.oriented.90.degree.sing},
    \begin{equation}  \label{E:diam.B.bound}
      diam(B) \geq 2 \arcsin \bigl( (1/2) \sqrt{  (k + 1)/k } \bigr) .
    \end{equation}
This is true for any neighborhood $\clU$ of the $90^{\circ}$ singularity $Y$, whether \eqref{E:reg.coef.bnd} holds or not. In particular, as proposition \ref{P:diam.oriented.90.degree.sing} points out, there are 
$Y_{1}, Y_{2} \in \D'$ arbitrarily close to $Y$ with the property that 
$o[Y_{1}] \cdot o[Y_{2}] < \cos(\pi3) = 1/2$. 

Let $Y_{1}, Y_{2} \in \clU \cap \D'$. Use $\epsilon$ flexibly and generically for numbers 
in $(0,1)$ that all go to 0 as any one of them does. 
Let $\alpha := \arcsin \bigl( (1/2) \sqrt{ (k + 1)/k }$. Then for 
$\epsilon > 0$ given we may assume 
    \begin{equation}  \label{E:angle.o[Yi]}
      \angle \bigl( o[Y_{1}], o[Y_{2}] \bigr) > 2 \alpha - \epsilon .
    \end{equation} 
We have
    \begin{align*}
      \cos (2 \alpha - \epsilon) &= 
              \cos (2 \alpha) \cos \epsilon + \sin (2 \alpha) \sin \epsilon \\
          &< \cos (2 \alpha) + \epsilon = \cos^{2} \alpha - \sin^{2} \alpha + \epsilon \\
          &= (1 - \sin^{2} \alpha) - \sin^{2} \alpha + \epsilon \\
          &= 1 - 2 \sin^{2} \alpha + \epsilon \\
          &= 1 - 2 \left(\frac{1}{4} \frac{k+1}{k} \right) + \epsilon \\
          &= \frac{k-1}{2k} + \epsilon < \frac{1}{2} + \epsilon .
    \end{align*}
Let $b_{i} := b(Y_{i})$ so, by \eqref{E:o[Y].defn}, 
$o[Y_{i}] = \bigl( |b_{i}|^{2} + 1 \bigr)^{-1/2} ( b_{i}^{T}, -1 )$ 
($i = 1,2$). Therefore, by the preceding and \eqref{E:angle.o[Yi]} in any neighborhood of the $90^{\circ}$ singularity $Y$ there are data sets $Y_{1}, Y_{2} \in \D'$ s.t.\
    \begin{equation*}
       \epsilon + \frac{k-1}{2k} > o[Y_{1}] \cdot o[Y_{2}]   
         = \frac{b_{1} \cdot b_{2} + 1}
           {\sqrt{\bigl( |b_{1}|^{2} + 1 \bigr) \bigl( |b_{2}|^{2} + 1 \bigr) }} .
    \end{equation*}
Hence,
    \begin{equation}  \label{E:dot.prod.of.coef.vecs}
       b_{1} \cdot b_{2} 
         < \frac{k-1}{2k} \sqrt{ \bigl( |b_{1}|^{2} + 1 \bigr) 
           \bigl( |b_{2}|^{2} + 1 \bigr) } - 1 + \epsilon.
    \end{equation}
Therefore,
    \begin{align}  \label{E:|b1-b2|.sqrd.lwr.bound}
      |b_{1} - b_{2}|^{2} &=  |b_{1}|^{2} - 2 b_{1} \cdot b_{2} + |b_{2}|^{2} \notag \\
           &> |b_{1}|^{2} 
             - \frac{k-1}{k} \sqrt{ \bigl( |b_{1}|^{2} + 1 \bigr) \bigl( |b_{2}|^{2} + 1 \bigr) } 
               + 2 - \epsilon + |b_{2}|^{2} \notag \\
           &= \frac{k+1}{k} \sqrt{ \bigl( |b_{1}|^{2} + 1 \bigr) 
             \bigl( |b_{2}|^{2} + 1 \bigr) } \\
           & \qquad + \Bigl[ \bigl( |b_{1}|^{2} + 1 \bigr)
             - 2 \sqrt{ \bigl( |b_{1}|^{2} + 1 \bigr) \bigl( |b_{2}|^{2} + 1 \bigr) } 
               + \bigl( |b_{2}|^{2} + 1 \bigr) \Bigr] - \epsilon \notag \\
           &= \frac{k+1}{k} \sqrt{ \bigl( |b_{1}|^{2} + 1 \bigr) 
             \bigl( |b_{2}|^{2} + 1 \bigr) }
             + \bigl( \sqrt{|b_{1}|^{2} + 1} - \sqrt{|b_{2}|^{2} + 1} \, \bigr)^{2}
               - \epsilon \notag \\
           &\geq \frac{k+1}{k} \sqrt{ \bigl( |b_{1}|^{2} + 1 \bigr) \bigl( |b_{2}|^{2} + 1 \bigr) } 
             - \epsilon \notag .
    \end{align}
Thus,
    \begin{equation}  \label{E:|b1-b2|>1}
       |b_{1} - b_{2}| \geq \sqrt{(k+1)/k} - \epsilon > 1 .
    \end{equation}
Hence, for example, in fitting a line to bivariate data ($k=1$) the near a $90^{\circ}$ singularity, the coefficients may differ by almost $\sqrt{2}$. \eqref{E:|b1-b2|>1} also implies that in general
$|b_{1}| + |b_{2}| \geq |b_{1} - b_{2}| \geq \sqrt{(k+1)/k} - \epsilon > 1$. So $|b_{1}|$ 
or $|b_{2}|$ is at least 1/2.

\eqref{E:|b1-b2|.sqrd.lwr.bound} also implies.
    \begin{equation*}  \label{E:b1.b2.rel.diff}
      \frac{|b_{1} - b_{2}|^{2}}{|b_{1}||b_{2}|} 
        > \sqrt{ \bigl( 1 + |b_{1}|^{-2} \bigr) \bigl( 1 + |b_{2}|^{-2} \bigr) } > 1 .
    \end{equation*}
Thus, the length of the difference between $b_{1}$ and $b_{2}$ is at least equal to the geometric mean, $\sqrt{|b_{1}||b_{2}|}$, of the lengths. \emph{That may be a very conservative bound!}

But it seems that, \emph{a priori}, little can be said about $|b_{1}|$ and $|b_{2}|$. 
  
\section{General lower bound on $\dim \Ss$ in plane-fitting} \label{SS:gnrl.lwr.bnd.plane.fit}
As observed in remark \ref{R:dim.S.for LS}, LS does not satisfy \eqref{E:big.lwr.bound.in.plane.fit}. In this section we prove that if \textbf{hypothesis \ref{Hyp:S.cap.T.small}} of Theorem \ref{T:Phi.star.Hr.contains.Theta.star.Hr} fails in plane-fitting one can still get the lower bound
   \begin{equation}  \label{E:lower.lower.bound.in.plane.fitting}
      \dim \Ss \geq \kappa := nk + (k + 1)(\nvar - k) -1 .
   \end{equation}
Notice that this is less than or equal to the lower bound $d - 2 = n \nvar -2$ in proposition \ref{P:sing.codim.in.plane-fitting} and is strictly less than that bound unless 
$n = \nvar+1$ and $\nvar = k+1$. (See \eqref{E:n>nvar>k>0}.) Note that, by proposition \ref{P:collin.data.are.sings.of.LS} and lemma \ref{L:dim.set.of.collin.data.sets}, the bound \eqref{E:lower.lower.bound.in.plane.fitting} \emph{is achieved} by least squares regression with 
$m = \nvar - k = 1$. 

  \begin{remark}
Suppose a statistician is in the habit of looking at the data and making subjective judgments about how to fit a plane to them. Then, unless the statistician is so incompetent he or she frequently does not recognize exact planar fits, they have a personal plane-fitting singular set of dimension no smaller than the LHS of \eqref{E:lower.lower.bound.in.plane.fitting}. 
  \end{remark}

Let $\D := \Y$ (see \eqref{E:mcl.Y,defn}). 
Metrize $\D$ by the Frobenius or Euclidean norm, 
$\| \cdot \|$, \eqref{E:matrix.norm}. Then $\D$ is $\RR^{n \nvar}$ with the usual norm. 
In this section we prove \eqref{E:lower.lower.bound.in.plane.fitting} using proposition \ref{P:sing.dim.when.H.d-r.D.=.0}. At the end of the section we derive 
\cite[Theorem 2.2, p.\ 493]{spE95} as a corollary of our work here. The proof given in \cite{spE95} is much shorter, but here we show how the result can be proved using theorem \ref{T:Phi.star.Hr.contains.Theta.star.Hr} and proposition \ref{P:sing.dim.when.H.d-r.D.=.0}. Moreover, this proof might generalize.

Let $\Pf$ be as in \eqref{E:Perfect.Fits.in.plane.fitting}. By lemma \ref{L:Pf.is.a.manif}, we have that $\Pf$ is a manifold of dimension $\dim \Pf = nk + (k + 1)(\nvar - k)$. Therefore, the inequality \eqref{E:lower.lower.bound.in.plane.fitting} is exactly \eqref{E:general.fallback} with $\T = \Pf$ and 
    \begin{equation*}
      r = 1. 
    \end{equation*}
(See \eqref{E:r=1}.) However, this way of invoking \eqref{E:general.fallback} does not work here because $\Pf$ is not a compact manifold. This means proposition \ref{P:sing.dim.when.H.d-r.D.=.0} cannot be used to prove \eqref{E:codim.S.leq.r+1} on which \eqref{E:general.fallback} depends. However, in this section we develop a way to prove \eqref{E:lower.lower.bound.in.plane.fitting} by making \eqref{E:general.fallback} valid to in the plane-fitting setting. 

The general plan is as follows. Let $y^{1 \times \nvar} \in \RR^{\nvar}$, 
    \begin{multline} \label{E:Dy.is.one.pt.compacitification}
      \text{Let } \widehat{\D}_{(y)} \text{ be the one point compactification of the space, }
        \D_{(y)}, \text{ of all } \\
          n \times \nvar \text{ matrices whose last row is } y.
    \end{multline} 
(See section \ref{SS:D.T.plane.fit}.)
Thus, $\widehat{\D}_{(y)}$ is homeomorphic 
to an $(n-1) \nvar$-sphere. Put on $\widehat{\D}_{(y)}$ the usual differentiable structure of a sphere. Denote the point at infinity of $\widehat{\D}_{(y)}$ by $\infty_{(y)}$. Parametrize $\widehat{\D}_{(y)} \setminus \{ \infty_{(y)} \}$ by stereographic projection (section \ref{SS:D.T.plane.fit}). As explained in section \ref{SS:D.T.plane.fit}, identifying $\Y$ with its image under inverse stereographic projection, $\Y$ is an imbedded submanifold of $\widehat{\D}_{(y)}$.
 
$\infty_{(y)} \in \widehat{\D}_{(y)}$ has dimension 0. Therefore, as explained in section \ref{SS:D.T.plane.fit} pertaining to $\D_{\infty}$,
if $\mcl{A} \subset \widehat{\D}_{(y)}$ is not $\{ \infty \}$, then, w.r.t.\ the metric in $\widehat{\D}_{(y)}$, $\dim \mcl{A}$ is the same as $\dim (\mcl{A} \cap \Y)$ as a subset of $\Y$ with $\Y$ metrized by the Frobenius norm. I.e., if $\mcl{A} \neq \{ \infty_{(y)} \}$ then removing the point at infinity does not change the dimension.

Let $\epsilon > 0$ be small. In $\widehat{\D}_{(y)}$ we construct a family, $\T^{\epsilon}_{(0)}$, of high-dimensional $\T$'s and apply proposition 
\ref{P:sing.dim.when.H.d-r.D.=.0} and \eqref{E:general.fallback} to $\T^{\epsilon}_{(0)}$ 
for each $\epsilon$. Integrating over $\epsilon$ raises the dimension by 1. Integrating over $y$ raises the dimension by $\nvar$ and we end up with the bound in \eqref{E:lower.lower.bound.in.plane.fitting}. 

For now take $y = 0^{1 \times \nvar}$. (Later we will let $y$ vary.) Define a submanifold 
of $\widehat{\D}_{(0)}$ as follows. For every $\xi \in G(k,\nvar)$, 
let $\Pi_{\xi}^{\nvar \times \nvar}$ be the matrix of orthogonal projection onto $\xi$. By \eqref{E:proj.mat.symmetric.idempotent},
	\begin{equation}  \label{E:basic.props.of. Pi}
		\Pi_{\xi} \text{ is symmetric and idempotent } (\Pi_{\xi}^{2} = \Pi_{\xi}). 
	\end{equation}
Moreover, the eigenvalues of $\Pi_{\xi}$ are either 0 or 1. Now let
	\begin{equation} \label{E:Y.xi.defn}
		Y_{\xi}^{n \times \nvar} :=
			\begin{pmatrix}
				\Pi_{\xi}^{\nvar \times \nvar} \\
				0^{(n-\nvar) \times \nvar}
			\end{pmatrix}
		\in \D_{(0)}.
	\end{equation}
Write
	\[
		\Y_{G(k, \nvar)} := \bigl\{ Y_{\xi} \in \Y : \xi \in G(k, \nvar) \bigr\}.
	\]
By considering the diagonalization of $\Pi_{\xi}$, we see that $\Y_{G(k, \nvar)}$ is a bounded subset (w.r.t.\ the Frobenius norm). 
	
By lemma \ref{L:proj.mat.is.imbedding.of.Grass}, we have
	\begin{equation}  \label{E:xi.to.Y.xi.imbed}
		\text{The map } P: \xi \mapsto Y_{\xi} \text{ is a smooth imbedding of } 
		         G(k, \nvar) \text{ into } \widehat{\D}_{(0)}. 
	\end{equation}

We construct a submanifold of $\widehat{\D}_{(0)}$ that is the total space of a sphere bundle (Milnor and Stasheff \cite[p.\ 38]{jwMjdS74}) over $\Y_{G(k, \nvar)}$ and, hence, 
over $G(k,\nvar)$. Let $\xi \in G(k, \nvar)$. Let $D^{k \times \nvar}$ be a matrix whose rows are orthonormal vectors 
in $\xi$. Denote the set of all such $D$ by $O_{\xi}$. Hence, by \eqref{E:proj.mat.from.K}, 
	\begin{equation}  \label{E:D.has.orthon.rows}
		D^{T} D = \Pi_{\xi} \text{ and } (D \Pi_{\xi} D^{T})^{k \times \nvar} 
		  = D D^{T} = I_{k}.
	\end{equation}
Let 
	\begin{equation}  \label{E:Vk.defn}
		\mcl{V}_{k} \text{ be the set of all } k \times \nvar 
		  \text{ matrices whose rows are orthonormal.}
	\end{equation} 
E.g., we have $O_{\xi} \subset \mcl{V}_{k}$. Given $D \in \mcl{V}_{k}$, 
let $U_{D} \subset G(k,\nvar)$ consist of $\zeta \in G(k, \nvar)$ s.t.\ the smallest eigenvalue of $D \Pi_{\zeta} D^{T}$ is strictly bigger than 1/2. $U_{D}$ is open by \eqref{L:proj.mat.is.imbedding.of.Grass} and lemma \ref{L:Eigen.cont.}. By \eqref{E:D.has.orthon.rows}, $\xi \in U_{D}$. I.e., $U_{D}$ is an open neighborhood of $\xi$. Obviously,  
	\begin{equation}   \label{E:D.Pi.DT.is invrtble}
		D \, \Pi_{\zeta} \, D^{T} \text{ is an invertible } k \times k \text{ matrix for every } 
		  \zeta \in U_{D}.
	\end{equation}

Let $\epsilon > 0$ and let 
	\begin{multline} \label{E:T.eps.defn}
		  \T^{\epsilon}_{(0)}
			:= \Bigl\{ Y \in \D_{(0)} : \text{ There exists } \xi \in G(k, \nvar) \\
			   \text { s.t.\ the rows of } Y \text{ lie exactly on } \xi  
			     \text{ and } \| Y - Y_{\xi} \| = \epsilon \Bigr\}.
	\end{multline} 
There might be multiple planes in $G(k,\nvar)$ on which the rows of 
$Y \in \T^{\epsilon}_{(0)}$ lie. Now, $\Y_{G(k, \nvar)}$ is a bounded set and any point 
of $\T^{\epsilon}_{(0)}$ is within $\epsilon$
 of $\Y_{G(k, \nvar)}$. Hence, 
    \begin{equation}  \label{E:T.eps.bdd}
      diam(T^{\epsilon}) \leq diam(\Y_{G(k, \nvar)}) + 2 \epsilon < + \infty .
    \end{equation}  

If $Y^{n \times \nvar} \in \T^{\epsilon}_{(0)}$, then $rank \, Y \leq k$, since $\rho(Y)$ is a subspace of a $k$-dimensional linear subspace of $\RR^{\nvar}$. 
For $Y \in \T^{\epsilon}_{(0)}$ and $\xi \in G(k, \nvar)$ as in \eqref{E:T.eps.defn}, call $\xi$ a ``corresponding plane'' for $Y$. $Y$ lies on the sphere in $\widehat{\D}_{(0)}$ centered at $Y_{\xi}$ with radius $\epsilon$.

\emph{Claim:} $\T^{\epsilon}_{(0)}$ is compact. We just saw that $\T^{\epsilon}_{(0)}$ is bounded 
w.r.t.\ $\| \cdot \|$ and $\T^{\epsilon}_{(0)} \subset \D_{(0)} \homeomto \RR^{(n-1) \nvar}$. (See \eqref{E:Dy.is.one.pt.compacitification}.) So by Heine-Borel (Rudin \cite[Theorem 2.41, p.\ 35]{wR64.PMA}), the closure, $\overline{\T^{\epsilon}_{(0)}}$, 
of $\T^{\epsilon}_{(0)}$ in $\D_{(0)}$ is compact. Let $Y^{n \times \nvar} \in \overline{\T^{\epsilon}_{(0)}}$. Then there exists a sequence
$\{ Y_{m} \} \subset \T^{\epsilon}_{(0)}$ s.t.\ $\| Y_{m} - Y \| \to 0$ as $m \to + \infty$. 
Let $\xi_{m} \in G(k, \nvar)$ be a corresponding plane for $Y_{m}$. Then for every $m$, 
we have $\| Y_{m} \Pi_{\xi_{m}} - Y_{m} \| = 0$ and $\| Y_{m}  - Y_{\xi_{m}} \| = \epsilon$. Since $G(k,\nvar)$ is compact there is a subsequence, $\{ \xi_{m} \}$ converging 
to some $\xi \in G(k, \nvar)$. By lemma \ref{L:proj.mat.is.imbedding.of.Grass} again and \eqref{E:Y.xi.defn}, the functions $(Y', \xi') \mapsto \| Y' \Pi_{\xi'} - Y' \|$ 
and $Y' \mapsto \| Y'  - Y'_{\xi'} \|$ are continuous in $Y' \in \D_{(0)}$ 
and $\xi' \in G(k, \nvar)$. 
Therefore, we must have $\| Y \Pi_{\xi} - Y \| = 0$ and $\| Y  - Y_{\xi} \| = \epsilon$. 
Thus, $Y \in \T^{\epsilon}_{(0)}$. I.e, $\T^{\epsilon}_{(0)} = \overline{\T^{\epsilon}_{(0)}}$. This proves the claim that $\T^{\epsilon}_{(0)}$ is compact.

By \eqref{E:T.eps.bdd}, $T^{\epsilon}$ \emph{is bounded} uniformly 
in $\epsilon \in (0, \nvar^{-2})$. \emph{Claim:} 
	\begin{equation}  \label{E:eps.<.nvar.-2.means.rank.k}
		\text{If } \epsilon < \nvar^{-2} \text{ and } Y^{n \times \nvar} \in \T^{\epsilon}_{(0)} 
		        \text{ then } rank \, Y \geq k \text{, i.e., } rank \, Y = k.
	\end{equation} 
(Because $\dim \rho(Y) \leq k$ for $Y \in \T^{\epsilon}_{(0)}$.) 
Let $Y^{n \times \nvar} \in \T^{\epsilon}_{(0)}$. Suppose 
    \begin{equation*}
      \epsilon \in (0, \nvar^{-2}).
    \end{equation*} 
Let $Y_{\nvar}^{\nvar \times \nvar}$ be the matrix consisting of the first $\nvar$ rows of $Y$. It suffices to show $rank \, Y_{\nvar} \geq k$. Let $\xi \in G(k, \nvar)$ be a corresponding plane for $Y$. Note that $Y_{\nvar} - \Pi_{\xi}$ is just the matrix of the first $\nvar$ rows of $Y - Y_{\xi}$. Let $z_{j}^{1 \times \nvar}$ be the $j^{th}$ row of $Y_{\nvar} - \Pi_{\xi}$ ($j=1, \ldots, \nvar$). Let $a^{\nvar \times 1}$ be a unit column-vector. Since $\| Y - Y_{\xi} \| = \epsilon$, we have that each entry in $Y - Y_{\xi}$, in particular each entry in $Y_{\nvar} - \Pi_{\xi}$, has absolute value no greater then $\epsilon$. Therefore,
	\begin{equation*}
		\bigl| (Y_{\nvar} - \Pi_{\xi}) a \bigr| = \bigl| (z_{1} a, \ldots, z_{\nvar} a)^{T} \bigr|
		  \leq \sum_{j=1}^{\nvar} | z_{j} a | \leq \sum_{j} |z_{j}| 
			 \leq \sum_{j} \nvar \epsilon \leq \nvar^{2} \epsilon < 1.
	\end{equation*}
Hence, if $a^{T} \in \xi$, so $\Pi_{\xi} a = a$, then
	\begin{equation*}
		| Y a | \geq | Y_{\nvar} a | 
		     \geq | \Pi_{\xi}  a | - \bigl| (Y_{\nvar} - \Pi_{\xi}) a \bigr|
	            \geq |a| - \nvar^{2} \epsilon = 1 - \nvar^{2} \epsilon > 0.
	\end{equation*}
Since $a$ is an arbitrary unit vector in a $k$-dimensional space, this proves the claim that $rank \, Y \geq k$. Let 
	\begin{equation}   \label{E:eps.<.1/nvar.sqrd}
	   \epsilon \in (0, \nvar^{-2}) \subset (0, 1/4) . 	
	\end{equation}
be arbitrary. Then $\epsilon < 1/4$ by \eqref{E:n>nvar>k>0}.

Thus, by \eqref{E:eps.<.nvar.-2.means.rank.k}, if $Y \in \T^{\epsilon}_{(0)}$ then 
$\rho(Y) \in G(k,\nvar)$. \emph{Claim:} 
    \begin{equation} \label{E:rho.surj.on.T.eps}
      \rho : \T^{\epsilon}_{(0)} \to G(k, \nvar) \text{ is surjective.}
    \end{equation}
To see this let $\xi \in G(k, \nvar)$. Let $D^{k \times \nvar} \in O_{\xi}$. Then, by \eqref{E:D.has.orthon.rows} and \eqref{E:basic.props.of. Pi}, 
$\Pi_{\xi}^{T} \Pi_{\xi} = \Pi_{\xi} = D^{T} D$. Now, 
$trace \, D^{T} D = trace \, D D^{T} = k$. Thus, $\| \Pi_{\xi} \| = \sqrt{k}$. Let $Y := (1+\epsilon/\sqrt{k}) Y_{\xi}$. By \eqref{E:Y.xi.defn}, $\rho(Y) = \xi$. Moreover,
$\| Y - Y_{\xi} \| = (\epsilon/\sqrt{k}) \| Y_{\xi} \| = (\epsilon/\sqrt{k})) \| \Pi_{\xi} \| = \epsilon$. I.e., $Y \in \T^{\epsilon}_{(0)}$, by \eqref{E:T.eps.defn}, and $\rho(Y) = \xi$. 
But $\xi \in G(k,\nvar)$ is arbitrary. This proves the claim \eqref{E:rho.surj.on.T.eps}.

Let 
    \begin{equation} \label{E:Pf.(0).defn}
      \Pf^{(0)} := \bigl\{ Y \in \D_{(0)}: rank \, Y = k \bigr\} 
        = \bigl\{ Y \in \D_{(0)}: \rho(Y) \in G(k,\nvar) \bigr\} .
    \end{equation} 
(Don't confuse $\Pf^{(0)}$ with $\Pf^{0}$ in remark \ref{R:degenerate.data}.) Note that by \eqref{E:when.is.Y.in.Pk} with $w^{n \times 1} := (0, \ldots, 0, 1)^{T}$ we see that $\Pf^{(0)} \subset \Pf^{k}$ as defined in \eqref{E:Perfect.Fits.in.plane.fitting} and, by \eqref{E:Delta(Y).defn},  
    \begin{equation} \label{E:rho.=.Delta.on.Pf.0}
        \text{If } Y \in \Pf^{(0)} \text{ then } \Delta(Y) = \rho(Y).
    \end{equation}
Any $Y \in \Pf^{(0)}$ can be written in the form  
	$\begin{pmatrix}
		X^{(n-1) \times \nvar} \\
		0^{1 \times \nvar}
	\end{pmatrix}$,
where $X$ has rank $k$. And conversely any such   
	$\begin{pmatrix}
		X^{(n-1) \times \nvar} \\
		0^{1 \times \nvar}
	\end{pmatrix}$ 
is in $\Pf^{(0)}$. Therefore, by lemma \ref{L:rank.k.mats.form.manif} (with ``$n-1$'' in place of ``$n$''), we have 
	\begin{equation}  \label{E:Pf(0).is.diff.manif}
		\Pf^{(0)} \text{ is an imbedded smooth submanifold of } \widehat{\D}_{(0)} 
		  \text{ of dimension } k(n-1) + k(\nvar-k).
	\end{equation}

By \eqref{E:T.eps.defn}, \eqref{E:eps.<.nvar.-2.means.rank.k}, and 
\eqref{E:eps.<.1/nvar.sqrd}, we have
	\begin{equation}  \label{E:T.eps.in.P.0}
		\T^{\epsilon}_{(0)} \subset \Pf^{(0)} .
	\end{equation} 
 
 \emph{Claim:}
	\begin{multline}  \label{E:T.eps.is.manif}
		\T^{\epsilon}_{(0)} \text{ is a compact imbedded differentiable submanifold of } 
		  \Pf^{(0)} \text{ of dimension }  \\
		    \dim \Pf^{(0)} - 1 = k(n-1) + k(\nvar-k) - 1 \in (1, \dim \D_{(0)}),
	\end{multline}
by \eqref{E:n>nvar>k>0}. We already know that $\T^{\epsilon}_{(0)}$ is compact. 
Let $\ell : \Pf^{(0)} \to \RR$ be defined by 
	\begin{equation}  \label{E:ell(Y).defn}
		\ell(Y) := \| Y - Y_{\rho(Y)} \|^{2}, \quad Y \in \Pf^{(0)}.
	\end{equation} 
By \eqref{E:T.eps.defn} and \eqref{E:T.eps.in.P.0}, if $Y \in \T^{\epsilon}_{(0)}$ its ``corresponding plane'' is just $\rho(Y)$. Hence, 
	\begin{equation} \label{E:ell.invrs.eps2.T.eps}
		\T^{\epsilon}_{(0)} =  \{ Y \in \Pf^{(0)} : \ell(Y) = \epsilon^{2} \} . 
	\end{equation} 
By \eqref{E:rho.is.smooth} and \eqref{E:xi.to.Y.xi.imbed}, the composition 
$Y \overset{\rho}{\mapsto} \rho(Y) \overset{P}{\mapsto} Y_{\rho(Y)}$ 
is smooth. Therefore, 
    \begin{equation}  \label{E:ell.is.diff'ble.on.P^(0)}
      \ell \text{ is differentiable.}
    \end{equation} 
Moreover, $\ell$ has rank 1. 
(Consider the function $t \mapsto \ell(tY)$, $t > 0$.) Therefore, by 
\eqref{E:Pf(0).is.diff.manif} and Boothby \cite[Theorem (5.8), p.\ 79]{wmB75}, 
we have that the claim \eqref{E:T.eps.is.manif} holds. 
Therefore, \textbf{hypothesis \ref{Hyp:T.manif}} of theorem \ref{T:Phi.star.Hr.contains.Theta.star.Hr} holds with $\T = \T^{\epsilon}_{(0)}$ 
and $t := \dim \T^{\epsilon}_{(0)} = k(n-1) + k(\nvar-k) - 1$.
It follows from \eqref{E:xi.to.Y.xi.imbed} that 
    \begin{equation}  \label{E:T.eps.cmpct.imbed.manif}
      \T^{\epsilon}_{(0)} \text{ is a compact imbedded submanifold of } \widehat{\D}_{(0)} .
    \end{equation}

Let 
    \begin{equation}  \label{E:pi.is.restrict.of.rho.to.T.eps}
      \pi := \rho \restriction_{\T^{\epsilon}_{(0)}}, 
    \end{equation}
the restriction of $\rho$ to $\T^{\epsilon}_{(0)}$. By \eqref{E:T.eps.is.manif}
and \eqref{E:xi.to.Y.xi.imbed}, $\T_{\epsilon}$ is an imbedded differentiable submanifold 
of $\widehat{\D}_{(0)}$. 

For the proof of the following see appendix \ref{Chptr:misc.proofs}.
	\begin{equation} \label{E:T.eps.is.sphere.bndl}
	  \bigl( \T^{\epsilon}_{(0)}, G(k,\nvar), S^{(n-1)k-1}, \pi \bigr)
	    \text{ is a fiber bundle over } G(k, \nvar) .
	\end{equation}
(See Spanier \cite[p.\ 90]{ehS66}.) Thus, $\T^{\epsilon}_{(0)}$ is the total space, $G(k,\nvar)$ is the base space, the fiber is the sphere $S^{(n-1)k-1}$, and $\pi$ is the bundle projection. Since we do not claim that the orthogonal group is the structure group of the bundle, technically we cannot claim that $\bigl( \T^{\epsilon}_{(0)}, G(k,\nvar), S^{(n-1)k-1}, \pi \bigr)$ is a sphere bundle. See Spanier \cite[p.\ 91]{ehS66}.

Let $\Phi: \Y \partlyto G(k, \nvar)$ be a plane-fitter. Thus, $\Phi$ is a data map satisfying \eqref{E:plane-fitter.defn}. Let $\Ss$ be the singular set of $\Phi$. If $\Ss$ is not closed,
we may apply the severity trick, remark \ref{R:severity.trick}, as described just before proposition \ref{P:sing.codim.in.plane-fitting} 
to allow us to replace $\Ss$ by the closed set $\Ss' := \Ss^{\msf{V}} \subset \Ss$ and assume $\Phi$ is continuous off $\Ss'$. Thus, $\dim \Ss \geq \dim \Ss'$. Let 
    \begin{equation} \label{E:Y'=Y-S}
      \Y' := \Y \setminus \Ss' .
    \end{equation}
If $\Ss'$ has nonempty interior, we are done. So assume $\Y'$ is dense in $\Y$ in accordance with \eqref{E:D'.=.D.less.S} (with $\D = \Y$, $D' = \Y'$).  
Recall that $\infty_{(y)}$ is the point at infinity of $\widehat{\D}_{(y)}$. 
For $y^{1 \times \nvar} \in \RR^{\nvar}$, let 
    \begin{equation}  \label{E:S'(y).defn}
      \Ss'_{(y)} :=  (\Ss' \cap \D_{(y)}) \cup \{ \infty_{(y)} \}.
    \end{equation} 
Then 
    \begin{equation}  \label{E:Ss'_(y).closed}
      \Ss'_{(y)} \text{ is closed }
    \end{equation} 
and the restriction 
$\Phi_{(y)} := \Phi \restriction_{\widehat{\D}_{(y)}}$ 
of $\Phi$ to $\widehat{\D}_{(y)}$ ($\Phi$ is not defined at $\infty_{(y)}$) is defined and continuous on $\widehat{\D}_{(y)} \setminus \Ss'_{(y)}$. Let 
	\begin{equation}  \label{E:s.n.k.nvar.defn}
		s := \dim \T^{\epsilon}_{(0)} - 1 =  k(n-1) + k(\nvar-k) - 2 > 0,
	\end{equation}
by \eqref{E:T.eps.is.manif}. We will prove that 
	\begin{equation}   \label{E:dim.S0.lwr.bnd.in.D0}
		\dim (\Ss' \cap \D_{(0)}) = \dim \Ss'_{(0)} \geq \dim \Pf^{(0)} - 1 
		  = \dim \T^{\epsilon}_{(0)} = s + 1.
	\end{equation}
The first equality holds because $\Ss' \cap \D_{(0)}$ and $\Ss'_{(0)}$ only differ by one point, $\infty_{(0)}$.
 
If $\Ss'_{(0)} \subset \widehat{\D}_{(0)}$ has non-empty interior relative 
to $\widehat{\D}_{(0)}$ then, 
$\dim \Ss'_{(0)} = \dim (\widehat{\D}_{(0)}) = (n-1) \nvar > s+1$. I.e., \eqref{E:dim.S0.lwr.bnd.in.D0} holds. So suppose $\Ss'_{(0)}$ has empty interior (relative to $\widehat{\D}_{(0)}$). Moreover, by \eqref{E:Ss'_(y).closed} $\Ss'_{(y)}$ is closed so, by \eqref{E:T.eps.cmpct.imbed.manif}, \textbf{hypothesis \ref{Hyp:S.cap.T.closed}} 
of theorem \ref{T:Phi.star.Hr.contains.Theta.star.Hr}  holds, too.

Similarly, if $\Ss'_{(0)} \cap \T^{\epsilon}_{(0)}$ has non-empty interior relative 
to $\T^{\epsilon}_{(0)}$ 
then $\dim \Ss'_{(0)} \geq \dim \T^{\epsilon}_{(0)}$ so then \eqref{E:dim.S0.lwr.bnd.in.D0} again holds. So suppose $\Ss'_{(0)} \cap \T^{\epsilon}_{(0)}$ has empty interior 
(relative to $\T^{\epsilon}_{(0)}$). \emph{A fortiori} we may similarly assume 
$\Pf^{(0)} \setminus \Ss'_{(0)}$ is dense in $\Pf^{(0)}$.  
$\T^{\epsilon}_{(0)} \setminus \Ss'_{(0)} \subset \Y'$, by \eqref{E:Y'=Y-S}. Therefore, by  \eqref{E:plane-fitter.defn}, \eqref{E:T.eps.in.P.0}, and \eqref{E:rho.=.Delta.on.Pf.0}, 
    \begin{equation} \label{E:when.Phi.=.Delta.on.T.eps}
        \Phi = \Delta = \rho \text{ on } \T^{\epsilon}_{(0)} \setminus \Ss'_{(0)}.
    \end{equation} 
Thus, 
	\begin{equation} \label{E:Phi.Y.=.rho.Y}
		\Phi(Y) = \rho(Y) \text{ for all } Y \in \T^{\epsilon}_{(0)} 
		  \text{ on a dense subset of }  \T^{\epsilon}_{(0)}.
	\end{equation} 
Therefore, by \eqref{E:convergence.in.Grassmann}, \textbf{hypothesis \ref{Hyp:extend}} of the theorem holds, too, with $\Theta = \pi := \rho \restriction_{\T^{\epsilon}_{(0)}}$. 

By \eqref{E:T.eps.is.manif} and \eqref{E:s.n.k.nvar.defn}, 
\textbf{hypothesis \ref{Hyp:T.manif}} of theorem \ref{T:Phi.star.Hr.contains.Theta.star.Hr} holds with $\T := \T^{\epsilon}_{(0)}$ 
and $t := \dim  \T^{\epsilon}_{(0)} = \dim \Pf^{(0)} - 1 = s+1 > 1$. Let  
    \begin{equation*}
      r = 1 
    \end{equation*} 
so \textbf{hypothesis \ref{Hyp:r.integer}} of theorem \ref{T:Phi.star.Hr.contains.Theta.star.Hr} holds.
\emph{Assume} that for some $\epsilon \in (0, \nvar^{-2})$ we have
	\begin{equation} \label{E:dim.S0.<.s}
		\dim( \Ss'_{(0)}  \cap \T^{\epsilon}_{(0)}) 
		  = \dim( \Ss'  \cap \T^{\epsilon}_{(0)})
		   < \dim \T^{\epsilon}_{(0)} - r = s . 
	\end{equation}
Then \eqref{E:dim.S0.<.s} implies \textbf{hypothesis \ref{Hyp:S.cap.T.small}} of theorem \ref{T:Phi.star.Hr.contains.Theta.star.Hr} holds holds for $\Phi = \Phi_{(0)}$ 
 and $\T = \T^{\epsilon}_{(0)}$. In appendix \ref{Chptr:misc.proofs} we show that \textbf{(\ref{E:nontriv.r-dim.homol})} holds
using two proofs suggested by Steven Ferry (personal communication). 

Thus, $(\Phi, \Ss' \cap \D_{(0)}, \T^{\epsilon}_{(0)}, \F)$ satisfies the \textbf{hypotheses} of theorem \ref{T:Phi.star.Hr.contains.Theta.star.Hr} and \eqref{E:nontriv.r-dim.homol}. Moreover, 
$\widehat{\D}_{(0)}$ is an $(n-1) \nvar$-sphere, and thus a compact manifold. Moreover, by \eqref{E:n>nvar>k>0}, $(n-1) \nvar \geq 4$. Therefore, with $r=1$ and $d := \dim \widehat{\D}_{(0)} = (n-1) \nvar$, we have $H^{d-r}( \widehat{\D}_{(0)} ) = \{0\}$.  
We know that $\Ss'$ is closed. Therefore, by proposition \ref{P:sing.dim.when.H.d-r.D.=.0}, if \eqref{E:dim.S0.<.s} holds then    
	\begin{equation}  \label{E:dim.S.cap.T.eps}
	    \Hm^{(n-1) \nvar - 2}(\Ss' \cap \D_{(0)}) > 0 . \text{ Therefore, } 
		\dim (\Ss' \cap \D_{(0)}) \geq (n-1) \nvar - 2 > s+1 , 
	   \end{equation}
by \eqref{E:s.n.k.nvar.defn}. So again \eqref{E:dim.S0.lwr.bnd.in.D0} holds. It is enough that \eqref{E:dim.S0.<.s} hold for just one $\epsilon$ satisfying 
\eqref{E:eps.<.1/nvar.sqrd}.
 
Suppose \eqref{E:dim.S0.<.s} does \emph{not} hold for \emph{any} $\epsilon$ satisfying 
\eqref{E:eps.<.1/nvar.sqrd} and let $J \subset (0, \nvar^{-2})$ be a closed interval 
 of positive length. Then, by \eqref{E:s.n.k.nvar.defn}, 
	\begin{equation} \label{E:failed.dim.S.cap.D(0).ineq}
		\dim (\Ss' \cap \D_{(0)}) \geq \dim( \Ss'_{(0)} 
		  \cap \T^{\epsilon}_{(0)}) \geq \dim \T^{\epsilon}_{(0)} - r = s
		    \text{ for every } \epsilon \in J.
	\end{equation}
 Thus, \textbf{hypothesis \ref{Hyp:S.cap.T.small}} of theorem \ref{T:Phi.star.Hr.contains.Theta.star.Hr} may not hold for $\Phi = \Phi_{(0)}$ 
 and $\T = \T^{\epsilon}_{(0)}$, for \emph{any} $\epsilon \in J$.  
So the argument that led to \eqref{E:dim.S.cap.T.eps} fails for every $\epsilon \in J$. 

The $s$ in \eqref{E:failed.dim.S.cap.D(0).ineq} is one less than the $s+1$ in \eqref{E:dim.S0.lwr.bnd.in.D0}. We now make up that difference. Let $\delta \in (0, s)$. 
Then $\dim( \Ss'_{(0)}  \cap \T^{\epsilon}_{(0)}) \geq s$ implies 
	\begin{equation}  \label{E.Hm.s.(S'.cap.T.eps.=.infty}
		\Hm^{s-\delta}( \Ss'_{(0)}  \cap \T^{\epsilon}_{(0)}) = + \infty 
		  \text{ for every } \epsilon \in J.
	\end{equation}
(See appendix \ref{Chptr:Lip.Haus.meas.dim}.) 

Let $\T^{J} = \cup_{\epsilon \in J} \T^{\epsilon}_{(0)} \subset \Pf^{(0)}$. By \eqref{E:T.eps.bdd}, 
$\T^{J}$ is bounded in $\D_{(0)} \homeomto \RR^{(n-1) \nvar}$. 
Define $\ell : \Pf^{(0)} \to \RR$ as in \eqref{E:ell(Y).defn}: 
$\ell(Y) := \| Y - Y_{\rho(Y)} \|^{2}$, $Y \in \Pf^{(0)}$. By \eqref{E:ell.invrs.eps2.T.eps}, 
$\T^{J} = \ell^{-1}(J^{2})$, where $J^{2} := \{ \epsilon^{2} : \epsilon \in J \}$. Note that $J$ is bounded away from 0. 
Let $\tilde{J} \subset (0, \nvar^{X-2})$ be an open interval containing $J$ whose left endpoint is positive. $\ell^{-1}(\tilde{J}^{2})$ is open in $\Pf^{(0)}$. By \eqref{E:ell.is.diff'ble.on.P^(0)}, $\ell$ is smooth on $\ell^{-1}(\tilde{J}^{2})$. Therefore, by corollary \ref{C:cont.diff.=.loc.Lip} and \eqref{E:local.Lip.is.Lip.on.compacts}, $\ell$ is Lipschitz on $\ell^{-1}(J^{2})$. Now, $J$ is bounded away from 0. This means $\ell^{1/2}$ is Lipschitz on $\ell^{-1}(J^{2})$. But $\T^{J} = \ell^{-1}(J^{2})$. Therefore, $\ell^{1/2}$ is Lipschitz on $\T^{J}$. 

Now, $J$ has positive length. By \eqref{E.Hm.s.(S'.cap.T.eps.=.infty} there do not exist any ``upper functions'' (Federer \cite[2.4.2, p.\ 81]{hF69} for 
$\epsilon \mapsto \Hm^{s-\delta}( \Ss'_{(0)}  \cap \T^{\epsilon}_{(0)})$. Therefore, 
$\int_{J}^{\ast} \Hm^{s-\delta}( \Ss'  \cap \T^{\epsilon}_{(0)}) \, d \epsilon = + \infty$ (ibid).
Therefore, by \eqref{E:ell.invrs.eps2.T.eps} and 
Federer \cite[Theorem 2.10.25, p.\ 188]{hF69}, there exists a constant 
$C < + \infty$ s.t.\ 
	\begin{equation*}
		+ \infty = \int_{J}^{\ast} \Hm^{s-\delta}( \Ss'  \cap \T^{\epsilon}_{(0)}) 
		  \, d \epsilon
		  = \int_{J}^{\ast} \Hm^{s-\delta} \bigl[ \Ss'  \cap (\ell^{1/2})^{-1}
		    (\epsilon) \bigr] 
		    \, d \epsilon \leq C \Hm^{s + 1-\delta}( \Ss'  \cap \T^{J}),
	\end{equation*}
where ``$\int^{\ast}$'' denotes upper integral (Federer \cite[2.4.2, p.\ 81]{hF69}). 
But $\delta \in (0, s)$ is arbitrary. From the definition of Hausdorff measure in appendix \ref{Chptr:Lip.Haus.meas.dim}, we conclude
	\begin{equation} \label{E:dim.S.cap.Pf(0).geq.s+1}
		\dim (\Ss' \cap \D_{(0)}) \geq \dim (\Ss' \cap \Pf^{(0)}) 
		  \geq \dim ( \Ss'  \cap \T^{J})  \geq s+1
	\end{equation}  
holds whether or not $\Ss'_{(0)}$ has nonempty interior. This completes the proof of \eqref{E:dim.S0.lwr.bnd.in.D0}. 

Now we drop the restriction to data matrices $Y^{n \times \nvar}$ whose last row is 0. Recall that $\Y$ is the space of all $n \times \nvar$ matrices. Metrize $\Y$ by the Frobenius norm, \eqref{E:matrix.norm}.  
Recall the definition \eqref{E:Dy.is.one.pt.compacitification} of $\D_{(y)}$. 
For every $y^{1 \times \nvar} \in \RR^{\nvar}$, 
let $\Phi_{(y)} = \Phi \restriction_{\D_{(y)}}$. 

Let $y^{1 \times \nvar} \in \RR^{\nvar}$. Let $h_{(y)} : \Y \to \Y$ be the map 
    \begin{equation}  \label{E:h_(y).defn}
      h_{(y)} : Y \mapsto Y + y 1^{n} \quad (Y \in \D). 
    \end{equation}
$h_{(y)}(Y)$ is the matrix $Y \in \Y$ with $y$ added to each row. $h_{(y)}$ is obviously an isometry and $h_{(y)}^{-1} = h_{(-y)}$. Recall the definition \eqref{E:Dy.is.one.pt.compacitification} of $\D_{(y)}$.
We have $h_{(-y)}(\D_{(y)}) = \D_{(0)}$. By \eqref{E:when.is.Y.in.Pk} with 
$w = n^{-1} 1^{n}$, we have that $h_{y}(\Pf) = \Pf$ and, by \eqref{E:Delta(Y).defn}, 
$\Delta \circ h_{(y)} = \Delta$ on $\Pf$. 

Let $\Psi(Y) := \Phi \circ h_{(y)} (Y)$, whenever it is defined. Recall the definition, 
\eqref{E:Y'=Y-S}, of $\Y'$. $\Psi$ is defined on $h_{(y)}^{-1}(\Y')$, which is dense in $\D$ since $\Y'$ is dense and $h_{y)}$ is an isometry. 
By \eqref{E:plane-fitter.defn}, $\Psi$ is a plane-fitter since $\Phi$ is one.  
Let $\Rcl = \D \setminus h_{(y)}^{-1}(\Y') \subset \D$. Then $\Rcl$ is a superset of the singular set of $\Psi$. Since $h_{(y)}^{-1}(\D) = \D := \Y$, by \eqref{E:g.commutes.w/.set.ops},
    \begin{multline*}
       \Rcl = \D \setminus h_{(y)}^{-1}(\Y') = \D \cap \bigl[ h_{(y)}^{-1}(\Y') \bigr]^{c}
         = h_{(y)}^{-1}(\D) \cap h_{(y)}^{-1} \bigl[ (\Y')^{c} \bigr] \\
           = h_{(y)}^{-1} \bigl[ \D \cap (\Y')^{c} \bigr]
             = h_{(y)}^{-1}(\D \setminus Y') = h_{(y)}^{-1} (\Ss') . 
    \end{multline*}
The argument that led to \eqref{E:dim.S0.lwr.bnd.in.D0} for $\Phi$ and $\Ss'$ applies equally as well to $\Psi$ and $\Rcl$:
	\begin{equation*} 
		\dim (\Rcl \cap \D_{(0)}) \geq s+1 ,
	\end{equation*}
where $s$ is defined in \eqref{E:s.n.k.nvar.defn}.  
Hence, by lemma \ref{L:loc.Lip.image.of.null.set.is.null},   
	\begin{equation} \label{E:dim.S.cap.D(y).lwr.bnd} 
	    \dim (\Ss' \cap \D_{(y)}) \geq \dim h_{(y)}^{-1} (\Ss' \cap \D_{(y)})
	      = \dim (\Rcl \cap \D_{(0)}) \geq s+1,
    			\quad ( y^{1 \times \nvar} \in \RR^{\nvar} ) 
	\end{equation}

Define 
    \begin{equation*}
      g_{n}(Y) := n^{th} \text { row of } Y \in \Y .
    \end{equation*} 
Thus, $\D_{(y)} = g_{n}^{-1}(y)$ so 
$\Ss' \cap \D_{(y)} = g_{n}^{-1}(y) \cap \Ss'$. Therefore, by \eqref{E:dim.S.cap.D(y).lwr.bnd}, 
$\dim \bigl[ \Ss' \cap g_{n}^{-1}(y) \bigr] \geq s+1$. 
Let $\D_{[1/2]} := \Bigl\{ Y \in \D : \bigl| g_{n}(Y) \bigr| < 1/2 \Bigr\}$. Note that $g_{n}$ 
is Lipschitz on $\D_{[1/2]}$ with Lipschitz constant 1. 

Define $B_{1/2}^{\nvar}(0) \subset \RR^{\nvar}$ as in \eqref{E:Euc.ball.defn}. 
Let $\overline{B}_{1/2}^{\nvar}(0)$ be its closure.Let 
    \begin{equation*}
      \mbf{B} := \overline{B}_{1/2}^{\nvar}(0) .
    \end{equation*}
If $\delta \in (0, s+1)$, by definition of Hausdorff measure (appendix \ref{Chptr:Lip.Haus.meas.dim}), 
$\Hm^{s+1-\delta} \bigl[ \Ss' \cap g_{n}^{-1}(y) \bigr] = + \infty$. 
Therefore, applying Federer \cite[Theorem 2.10.25, p.\ 188]{hF69} again as above,  
there is a constant $C < + \infty$ s.t., 
	\begin{equation*}
		+ \infty = \int_{\mbf{B}}^{\ast} \Hm^{s+1-\delta} 
		    ( \Ss' \cap \D_{(y)} ) \, d y
		  = \int_{\mbf{B}}^{\ast} \Hm^{s+1-\delta} 
		    \bigl[ \Ss' \cap g_{n}^{-1}(y) \bigr] \, \Hm^{\nvar}(d y)
		      \leq C \Hm^{s + 1 + \nvar-\delta}( \Ss' ) . 
	\end{equation*}
We conclude
	\begin{equation} \label{E:dim.S.geq.kappa}
		\dim \Ss' \geq s+\nvar + 1 = \kappa.
	\end{equation}
This proves \eqref{E:lower.lower.bound.in.plane.fitting}.

We can use the preceding to prove \cite[Theorem 2.2, p.\ 493]{spE95}: 
  \begin{prop} \label{P:dim(S.cap.P)<kappa}
Let $\Phi : \Y \partlyto G(k,\nvar)$ be a plane-fitter (i.e., a data map satisfying 
\eqref{E:plane-fitter.defn}) with singular set $\Ss$. 
Suppose $\dim (\Ss \cap \Pf) < \kappa$, 
where $\kappa$ is defined in \eqref{E:lower.lower.bound.in.plane.fitting}. 
Then $\dim \Ss \geq nq -2$. (Otherwise, $\dim \Ss \geq \kappa$, of course.)
  \end{prop}
Note that we only assume $\Phi$ satisfies \eqref{E:plane-fitter.defn}. With $\T = \Pf$ and 
$r = 1$, by lemma \ref{L:Pf.is.a.manif}, $\dim (\Ss' \cap \Pf) < \kappa$ becomes 
$\dim (\Ss' \cap \T) < \dim \T - r$, just as in \textbf{hypothesis \ref{Hyp:S.cap.T.small}} of theorem \ref{T:Phi.star.Hr.contains.Theta.star.Hr}. And, yes, $\Pf$ is a manifold. The difference is that, unlike $\T$ in \textbf{hypothesis \ref{Hyp:T.manif}} of the theorem, 
\emph{$\Pf$ is not compact.} So theorem \ref{T:Phi.star.Hr.contains.Theta.star.Hr} does not apply here.

The proof of proposition \ref{P:dim(S.cap.P)<kappa} we give is actually longer than the more elementary proof given in \cite{spE95}! But our proof shows that the theory developed in this book indeed generalizes the results in \cite{spE95} and perhaps points the way to further generatizations?

  \begin{remark}  \label{R:general.lwr.bnd.on.plane-fitting.sing.set.dim}
It seems that any non-regularized (remark \ref{R:regularization.generalities}) process, no matter how subjective or algorithmic, that claims the name ``plane-fitting'' will satisfy \eqref{E:plane-fitter.defn}. Thus, any such process will always have singularities and 
$\kappa$ will always be a lower bound on the dimension of the set of them. (See section \ref{S:topology}.)
  \end{remark}

  \begin{proof}[Proof of proposition \ref{P:dim(S.cap.P)<kappa}]
Suppose 
    \begin{equation}  \label{E:S.cap.Pf.and.S.both.have.small.dim}
      \dim (\Ss \cap \Pf) < \kappa \text{ but } \dim \Ss < nq -2.
    \end{equation} 
Suppose $\Ss$ is not closed. Then apply the severity trick as above to get a closed set $\Ss' \subset \Ss$ off which we may assume $\Phi$ is continuous. 
Then 
    \begin{equation}  \label{E:dimS'.dimS.kappa}
      \dim (\Ss' \cap \Pf) \leq \dim (\Ss \cap \Pf) < \kappa \text{ but } 
        \dim \Ss' \leq \dim \Ss < nq -2 .
    \end{equation} 

Recall that $\dim \D_{(y)} = (n-1) \nvar$. Then, 
by Federer \cite[Theorem 2.10.27, p.\ 190]{hF69},
    \begin{equation} \label{E:gn.invrs.int=0}
      \int_{\mbf{B}}^{\ast} \Hm^{(n-1) \nvar - 2} 
		    ( \Ss' \cap \D_{(y)} ) \, dy 
	      \leq C \Hm^{n \nvar - 2}( \Ss' ) = 0 .
    \end{equation}
(By \eqref{E:when.Haus.meas.=.Leb.meas}, $dy = \Hm^{\nvar}(dy)$.)  

Let $f(y) := \Hm^{(n-1) \nvar - 2} ( \Ss' \cap \D_{(y)} ) \geq 0$ 
($y \in \mbf{B}$). \eqref{E:gn.invrs.int=0} says 
$\int_{\mbf{B}}^{\ast} f(y) \, dy = 0$. By definition of upper integral 
(Federer \cite[p.\ 81]{hF69}), for every $m = 1, 2, \ldots$ there exists a step function 
$u_{m} \geq f$ s.t.\ $\int_{\mbf{B}} u_{m}(y) \, dy < 1/m$.\footnote{
$u_{m}$ is a step function: (1) $u_{m}$ is $\Hm^{\nvar}$-measurable. 
(Since $u_{m}$ is defined on $\mbf{B}$, this is the same as Lebesgue-measurable. See \eqref{E:when.Haus.meas.=.Leb.meas}.) And (2) $u_{m}$ takes on an at most countable set of values, in $\RR$. Define 
$\int_{\mbf{B}} u_{m}(y) \, dy 
:= \sum_{t \in \RR} t \Hm^{\nvar} \bigl( u_{m}^{-1}(t) \bigr)$. The sum is taken over 
$t \in u_{m} \bigl( \mbf{B} \bigr)$. See Federer \cite[p.\ 81]{hF69}.
} 
Since $f \geq 0$, each $u_{m}$ is non-negative. We may assume $u_{m} \downarrow$. 
Let $u_{\infty} := \liminf_{m \to \infty} u_{m} \geq f \geq 0$. Thus, $u_{\infty}$ is 
$\Hm^{\nvar}$-measurable (Federer \cite[p.\ 73]{hF69}) and by \eqref{E:gn.invrs.int=0} and Fatou's lemma (Federer \cite[2.4.6, p.\ 84]{hF69}), 
$\int_{\mbf{B}} u_{\infty}(y) \, dy = 0$. 
Let $Z := u_{\infty}^{-1}(0) \subset f^{-1}(0) \subset \mbf{B}$. 
Then $Z$ is $\Hm^{\nvar}$-measurable. \emph{Claim:} 
    \begin{equation*}
      \Hm^{\nvar}(Z) > 0 .
    \end{equation*}
Suppose not. 
For $j = 1, 2, \ldots$ let $A^{j} := \{ y \in \mbf{B} : u_{\infty}(y) > 1/j \}$. 
So $Z = \mbf{B} \setminus \bigcup_{j} A^{j}$. Since $\Hm^{\nvar}(Z) = 0$, there must exist $j$ s.t.\ $\Hm^{\nvar}(A^{j}) > 0$. That means
$\int_{\mbf{B}} u_{\infty}(y) \, dy \geq (1/j) \Hm^{\nvar}(A^{j}) > 0$. Contradiction. 
Therefore, $\Hm^{\nvar}(Z) > 0$ as claimed. 

Let $y^{1 \times \nvar} \in Z$. Recall the definition, \eqref{E:h_(y).defn}, of $h_{(y)}$. Suppose $\epsilon > 0$ satisfies \eqref{E:eps.<.1/nvar.sqrd}. Recall \eqref{E:Pf.invar.under.shifts.rescale}. Define $\Pf^{(0)}$ as in \eqref{E:Pf.(0).defn} and define
    \begin{equation}  \label{E:Pf.(y).defn}
      \Pf^{(y)} := h_{(y)}(\Pf^{(0)}) \subset \Pf , 
    \end{equation}
and $\T^{\epsilon}_{(y)} :=  h_{(y)}(\T^{\epsilon}_{(0)})$, where $\T^{\epsilon}_{(0)}$ is defined in \eqref{E:T.eps.defn}. Then, by \eqref{E:T.eps.in.P.0}, 
$\T^{\epsilon}_{(y)} \subset \Pf^{(y)}$. 

Arguing as in the proof of \eqref{E:dim.S.cap.T.eps},
we see that \eqref{E:dim.S0.<.s} with ``$(y)$'' in place of ``$(0)$'' implies \eqref{E:dim.S.cap.T.eps} with ``$(y)$'' in place of ``$(0)$''. But $y \in Z$, so \eqref{E:dim.S.cap.T.eps} (with ``$(y)$'' in place of ``$(0)$'') is false. Therefore, the corresponding variation on \eqref{E:dim.S0.<.s} is false. I.e.,  we have 
	\begin{equation*} 
	  \text{For every } \epsilon \in (0, \nvar^{-2}) ,\; 	  
	    \dim( \Ss' \cap \T^{\epsilon}_{(y)}) \geq s ,
	\end{equation*}
where $s$ is defined in \eqref{E:s.n.k.nvar.defn}. Hence, arguing as in the proof of \eqref{E:dim.S.cap.Pf(0).geq.s+1} (with 
$\ell(Y) := \| Y - 1^{n} y - Y_{\rho(Y - 1^{n} y)} \|^{2}$ for $Y \in \Pf^{(y)}$), we have
	\begin{equation}  \label{E:dim.S.cap.Pf(y).geq.s+1}
		\dim (\Ss' \cap \Pf^{(y)}) \geq s+1 .
	\end{equation} 
	
By hypothesis of the proposition, we may pick $\delta > 0$ so small that 
    \begin{equation}  \label{E:Hm.kappa-delta.(S.cap.Pf).=.0}
      \Hm^{\kappa-\delta}(\Ss' \cap \Pf) = 0 .
    \end{equation}
By \eqref{E:dim.S.cap.Pf(y).geq.s+1}, we have 
$\Hm^{s-\delta+1}(\Ss' \cap \Pf^{(y)}) = +\infty$. But $y \in Z \subset \mbf{B}$ is arbitrary and $\Hm^{\nvar}(Z) > 0$. Thus, 
    \begin{equation*}
      \int_{Z}^{\ast} \Hm^{s-\delta+1}(\Ss' \cap \Pf^{(y)}) \, \Hm^{\nvar}(dy) = + \infty .
    \end{equation*}
By \eqref{E:s.n.k.nvar.defn}, we have $s+\nvar + 1 = \kappa$. Therefore, by Federer \cite[Theorem 2.10.25, p.\ 188]{hF69} yet again,  
there is a constant $C < + \infty$ s.t.,
	\begin{multline*}
           + \infty = \int_{Z}^{\ast} \Hm^{s-\delta+1}(\Ss' \cap \Pf^{(y)}) \, \Hm^{\nvar}(dy) \leq 
           \int_{\mbf{B}}^{\ast} \Hm^{s-\delta+1}(\Ss' \cap \Pf^{(y)}) \, \Hm^{\nvar}(dy) \\
             \leq C \Hm^{s+\nvar-\delta+1} \bigl( \Ss' \cap \Pf \cap g_{n}^{-1}(\mbf{B})  \bigr)
                \leq C \Hm^{\kappa-\delta} (\Ss' \cap \Pf)= 0 .
        \end{multline*}  
This contradicts \eqref{E:Hm.kappa-delta.(S.cap.Pf).=.0} showng that \eqref{E:S.cap.Pf.and.S.both.have.small.dim} is impossible and proves the proposition.
  \end{proof}

Now $\F = G(k, \nvar)$ is a smooth manifold. Therefore, by propositions \ref{P:smooth.manifs.have.commutative.convex.combos} or \ref{P:make.conv.combo.fns.from.curves} we may apply the severity trick (remark \ref{R:severity.trick}) and conclude that $\F$ has an open covering $\msf{V}$ s.t.\: 
If $\Phi: \Y \partlyto G(k, \nvar)$ is a ``plane-fitter'', i.e.\ a data map satisfying 
\eqref{E:plane-fitter.defn}, then 
   \begin{equation*}  
      \dim \Ss^{\msf{V}} \geq nk + (k + 1)(\nvar - k) -1 .
   \end{equation*} 
  
\section{Final remarks}
  \begin{remark}
Plane-fitting is not a dimensionless enterprise. Rescaling the variables can change the severity of a singularity, but not the dimension of the singular set. 
  \end{remark}

  \begin{remark}[``Highly probable singularities'']   \label{R:dnsity.contrs.as.data.spaces}
   
Suppose $Z^{n \times \nvar}$ is a random data set, i.e.\ random point in $\Y$, with independent and identically distributed observations. I.e., the rows of $Z$ are random with the same distribution, but they are statistically independent of each other. Suppose further that the common probability distribution of the rows of $Z$ is absolutely continuous with continuous density on $\RR^{\nvar}$ (w.r.t.\ Lebesgue measure) $f$. Then the probability density of $Z$ is $f^{\otimes n} := f \otimes \cdots \otimes f$. I.e., given $Y \in \Y$ with rows $y_{1}, \ldots, y_{n} \in \RR^{k}$, we have $f^{\otimes n}(Y) = f(y_{1}) \cdots f(y_{n})$. 
Suppose $f$ is unimodal, i.e., has just one local, and therefore global, maximum, at $y_{0} \in \RR^{\nvar}$. Then $f^{\otimes n}$ has just one local, and therefore global, maximum, at $Y_{0}^{n \times \nvar} = 1^{n} y_{0}$ (see \eqref{E:1n.col.vec.defn}). Thus, every row of $Y_{0}$ equals $y_{0}$. 

Next, we assume the contours, i.e. level sets, of $f^{\otimes n}$ are homeomorphic 
to a $(n \nvar-1)$-sphere and starlike w.r.t.\ $Y_{0}$ in the sense that if $Y$ lies on or surrounded by a contour, then the line segment joining $Y_{0}$ to $Y$ is also surrounded by the contour. Then every contour has the form $\D_{\mu} + Y_{0}$ (see \eqref{E:plane.fitting.D.mu.defn} and remark \ref{R:shift.D.mu}) for some 
$\mu : S^{n \nvar-1} \to (0, \infty)$. 

An obvious example is $Z$ with rows that are independent, identically distributed (non-degenerate) multivariate Gaussian each with mean $y_{0}$ (Johnson and Wichern 
\cite[Chapter 4]{raJdwW92}). Classically in Statistics, such $Z$ is often taken as a model for multivariate data sets. Then the contours are $(n \nvar-1)$-dimensional ellipsoids.

Let $\Phi : \Y \partlyto G(k, \nvar)$ be a plane-fitter, i.e.\ it satisfies \eqref{E:plane-fitter.defn}, on $\D$. 
Let $\mcl{C} = \D_{\mu} + Y_{0} \subset \RR^{n \nvar}$ be an arbitrary contour of 
$f^{\otimes n}$. By remark \ref{R:shift.D.mu}, $\Phi$ may have singularities in $\mcl{C}$. I expect that typically proposition \ref{P:sing.codim.in.plane-fitting} will apply and $\Phi$ will have lots of severe singularities in $\mcl{C}$. For example, the set of collinear data sets (see \eqref{D:collinearity}) is invariant under scalar multiplication, so any ellipsoid around $Y_{0}$ will contain collnear data sets, i.e., singularities of least squares regression (proposition \ref{P:collin.data.are.sings.of.LS}). By remark \ref{Ex:all.LS.sings.are.90.degrees} all those singularities will be severe if $m = 1$. . 

But the preceding holds for any contour $\mcl{C}$ of $f$ and we can choose $\mcl{C}$ on which the value of $f$ is arbitrarily close to the maximum of $f^{\otimes n}$, \emph{viz.}, $f(y_{0})^{n}$. Thus, in this (plausible) situation the severe singularities of $\Phi$ are not confined to god forsaken -- i.e.\ low density ($f^{\otimes n}$) -- corners of $\D$. They are rare merely because $\Ss$ has positive codimension. This highlights the importance of dimension (subsection  \ref{SS:dim.of.sing.sets}) and measure (chapter \ref{Chptr:Haus.meas.of.sing.set}) as tools for studying the singular set.

Taking this idea to its logical conclusion, note that $f^{\otimes n}$ achieves its maximum value at $Y_{0}$. $Y_{0}$ is the ``most probably data set''. But for most (all?) plane-fitters used in practice, $Y_{0}$ is a singularity. I.e., the most probably data set is a singularity. 

A general, but perverse, way to make the same point is as follows. Suppose $\D$ is compact and the singular set $\Ss \subset \D$ is as well. Then there is a probability density on $\D$ proportional to $x \mapsto \exp \bigl( - dist(x, \Ss) \bigr)$. This way the singularities are precisely the most probable data sets in $\D$.
  \end{remark}
  
  \begin{remark}[Non-parametric regression]  \label{R:nonparam.reg}
 In regression, given a data set $Y = (X^{n \times k}, y^{n \times 1})$ as in section \ref{SS:lin.reg.and.LS} and a space $\mathrm{F}$ of functions from $\D$ to $\RR$, one computes an element $f = \mcl{L}(Y) \in \mathrm{F}$ s.t.\ $f(x_{i})$ is close to $y_{i}$, in some sense, for all (or at least most) $i = 1, \ldots, n$ (e.g., Ogden \cite[Section 2.2]{rtO97.wavelets}, Wahba \cite{gW90.spline}, Cucker and Smale \cite{fCsS02.learning}; see remarks \ref{R:learning.and.predicting} and subsection \ref{SS:functions.vs.geometry}). 
 
Let $\Phi$ be a mapping on $\D$ that captures much of what $\mcl{L}$ is doing. 
Denote the codomain of $\Phi$ by $\F$. In non-parametric regression, the function space $\mathrm{F}$ is often infinite dimensional. This means 
$\F$ is infinite dimensional. If $\mathrm{F}$ contains all affine functions then fitting functions in $\mathrm{F}$ might inherit the singularity issues described in this chapter. (This point was made in remark \ref{R:transformed.vars}.)

\emph{However,} (a) Non-parametric regression often involves regularization (remark \ref{R:regularization.generalities}) and regularized linear regression cannot be expected to have singularities. So reducing non-parametric regression to linear regression probably will not capture instability in $\Phi$ well. (b) Even if reducing to linear does work it seems there are many more ways non-parametric regression can be unstable than linear regression can. Even with regularization might that instability have a topological basis? 

Might our theory be extended to reveal singularity in nonparametric regression? Might nonlinear functional analysis (Brown \cite{rfB93.TopolNonlinAnlys}) help?   
  \end{remark}

  \begin{remark}  \label{R:designer.sing.sets}
Here is a rather speculative thought. One way to explain the discrepancy between the sizes of the singular sets of least squares (section \ref{SS:lin.reg.and.LS}) and LAD (section \ref{SS:LAD}) is as follows. The bound \eqref{E:codim.LAD.sing.set}, on the codimension of the singular set of LAD, also applies also to $90^{\circ}$ singularities of LAD. By remark \ref{Ex:all.LS.sings.are.90.degrees}, every singularity of LS is a $90^{\circ}$ singularity. Near a severe singularity, the discrepancy between the fitted regression plane and the ``true'' regression plane can be large (section \ref{SS:90.degree.sings.in.line-fitting.on.plane}). How serious a problem singularity is may depend on one's loss function. If the loss function, like squared error, heavily penalizes large errors and lightly penalizes small errors, then it pays to make the singular set small. On the other hand, if the loss function, like absolute error, penalizes large errors less heavily and penalizes small errors more heavily, then it might pay to sop up the variability of one's data map in a large singular set so that the data map can be less variable elsewhere. Singularity is not an absolute evil. (A similar point is made in section \ref{SS:measure.distance.to.P}.) If one is designing a data map, $\Phi$, one might choose to stow some of the variability of $\Phi$ in singularities.
  \end{remark}

\chapter{Location Problem for Directional Data}  \label{Chptr:spherical.location}
This is the first of three chapters on this topic. Let
    \begin{equation}  \label{E:nvar.n.for.spheres}
      \nvar, \, n \text{ be positive integers with } n > 1 .
    \end{equation}
Let $y_{1}, \ldots, y_{n}$ be points on the 
$\nvar$-sphere, $S^{\nvar} := \bigl\{ y \in \mathbb{R}^{\nvar +1} : |y| = 1 \bigr\}$ 
($| \cdot |$ = Euclidean norm). In this case we say that the ``sample size'' is $n$ and the points $y_{1}, \ldots, y_{n}$ are ``observations''. Consider the problem of measuring location of such data clouds on the sphere. 
(Fisher \emph{et al}  \cite{niFtLbjjE87} and Watson \cite{gW82} are general references.) For example, we consider analogues on the sphere of the sample mean (chapter \ref{Chptr:aug.direct.mean}) and median (chapter \ref{Chptr:robst.loc.on.circle}). (The general problem of computing a mean on a Riemannian manifold has been studied. 
See Karcher \cite{hK77.RiemannianCenterOfMass}, 
Bhattacharya and Patrangenaru \cite{rBvP02.LocationDispersionOnManifs}, \cite{rBvP05.MeansOnManifs}, Pigoli and Piercesare Secchi \cite{dPpS12.MeanOnRiemannianManif}, Bhattacharya and Patrangenaru \cite{rBvP14.StatOnManifs}, and Bhattacharya \cite{rB13.NonparamStatManifs}.)

The data space is the Cartesian product, 
    \begin{equation}   \label{E:D.defn.in.loc.prob.on.sphere}
      \D := (S^{\nvar})^{n}, 
    \end{equation}
so $d := \dim \D = n \nvar$. Let
    \begin{equation}  \label{E:N.sub.n}
      \NN_{n} := \{1, \ldots, n \} .
    \end{equation}
and
    \begin{equation}  \label{E:symbol.for.symmetric.group}
      \text{Let } S_{n} \text{ be the group of permutations of } \NN_{n}. 
    \end{equation}
Since $\D$ is a compact $C^{1}$ manifold 
(Boothby \cite[Theorem (1.7), p.\ 57]{wmB75}) 
$\D$ is triangulable (lemma \ref{L:biLip.triangulation}). Hence, by lemma \ref{L:biLip.triangulation}, there exists a finite simplicial complex $P$ and a bi-Lipschitz homeomorphism $f : |P| \to \D$. $S_{n}$ obviously acts on 
$\D := (S^{\nvar})^{n}$ by permutation of the factors. Therefore, by lemma \ref{L:Cart.pwrs.of.simp.cmplx.can.be.perm.invar.cmplx}, 
    \begin{multline} \label{E:prod.of.spheres.has.perm.invar.triangulation}
      \D \text{ has a bi-Lipschitz triangulation invariant under action of } G = S_{n} \\
        \text{ as required by theorem \ref{T:lwr.bnd.on.Haus.meas}.}
    \end{multline}

Regard $\D$ as a subset of $\RR^{n(\nvar +1)}$ and put 
on $\D$ the Riemannian metric it inherits from $\RR^{n(\nvar+1)}$. 
Thus, if $v = (v_{1}, \ldots, v_{n})$ 
and $w = (w_{1}, \ldots, w_{n})$ are tangent vectors on $\D$ 
(so $v_{i}, w_{i} \in \RR^{\nvar + 1}$) at $x \in \D$, then 
    \begin{equation}  \label{E:directional.location.Riem.metric}
        \langle v, w \rangle_{x} = v_{1} \cdot w_{1} + \cdots + v_{n} \cdot w_{n},
    \end{equation} 
where ``$\cdot$'' denotes ordinary Euclidean inner product on $\RR^{\nvar + 1}$. Therefore, geodesics on $\D$ are products of great circular arcs on each 
$S^{\nvar}$ factor in the sense any geodesic on $\D$ has the following form. 
Let $i = 1, \ldots, n$ and let $\gamma_{i} : [0,1] \to S^{\nvar}$ be a geodesic on $S^{\nvar}$. 
Thus, $\gamma_{i}$ maps $[0,1]$ onto a great circular arc on $S^{\nvar}$. Then 
    \begin{multline} \label{E:geod.on.prod.of.spheres}
        \gamma_{1} \times \cdots \times \gamma_{n} : 
          t \mapsto \bigl( \gamma_{1}(t), \ldots, \gamma_{n}(t) \bigr), \; t \in [a, b], \\
            \text{ where } \gamma_{i} \text{ parametrizes a great circular arc on } 
              S^{\nvar} \\
                \text{ w.r.t.\ a multiple of arc length, } i = 1, \ldots, n, \\ 
                  \text{ is the form of every geodesic on } \D .
    \end{multline}
For proof, see appendix \ref{Chptr:misc.proofs}. (See also \eqref{E:when.Gamma.is.geod.2}.) 
    \begin{equation}   \label{E:rho.is.top.metric}
      \text{Let } \rho \text{ be the topological metric on } \D 
        \text{ determined by the Riemannian metric on } \D. 
    \end{equation}
So in this chapter 
    \begin{equation} \label{E.rho.plays.role.of.xi.in.loc.on.spheres}
      \rho \text{ plays the role of $\xi$ in chapter \ref{SS:main.haus.theorem}.}
    \end{equation} 
(See \eqref{E:xi.is.metric.on.D}.)

Note that 
    \begin{equation}  \label{E:Riem.metric.on.prod.of.spheres.is.invar}
      \text{The Riemannian metric } \langle v, w \rangle_{x} 
        \text{ is invariant under the action of the group } S_{n}.
    \end{equation} 
(See \eqref{E:directional.location.Riem.metric} and \eqref{E:symbol.for.symmetric.group}.)
 
  \begin{remark}[Another metric on $\D$]  \label{R:another.metric}
Another, topologically equivalent, metric on $\D$ is defined as follows. Let $x,y \in D$ and let $\gamma_{1} \times \cdots \times \gamma_{n}$ be a shortest geodesic arc connecting them as in \eqref{E:geod.on.prod.of.spheres}. Then the length of each arc 
$\gamma_{i}$ is no greater then than $\pi$. Let $\chi(x,y)$ be the maximum length of the $\gamma_{i}$'s. An equivalent definition is as follows. Write 
$x = (x_{1}, \ldots, x_{n})$ and $y = (y_{1}, \ldots, y_{n})$ so 
$x_{i}, y_{i} \in S^{\nvar} \subset \RR^{\nvar+1}$. 
Then $\chi(x,y) = \max_{i=1, \ldots, n} \angle(x_{i}, y_{i})$. (See \eqref{E:angle.between.vectors}. \eqref{E:Delta.metric.for.loc.on.circle} defines another metric on $\D$.) 

By \eqref{E:n.c.sqrd.sum.ineq}, $\chi$ generates the same topology on $\D$ 
that $\rho$ does. But unless otherwise specified, any reference to distance in $\D$, by the word ``nearest'' for example, involves distance measured by $\rho$.
  \end{remark}

The feature space is just the sphere 
    \begin{equation}  \label{E:F.for.loc.on.sphere}
      \F := S^{\nvar} .
    \end{equation}
Take the test pattern space, $\T$, to be the ``diagonal,''.
	\begin{equation}  \label{E:directional.T.defn}
	  \T  := \T_{n} 
	    := \{ (y_{1}, \ldots, y_{n}) \in \D \; : \; y_{1} = \cdots = y_{n} \in S^{\nvar} \}.
	\end{equation}
Hence,
    \begin{equation*}
       \dim \T = \nvar .
    \end{equation*}

It is natural to take the space, $\Pf$, of perfect fits to also be the diagonal. I.e., $\Pf=\T$. (See chapter \ref{Chptr:robst.loc.on.circle} for another choice.) It is also reasonable to suppose that for the location problem, the standard, 
    \begin{equation}  \label{E:Sigma.for.loc.on.sphere}
       \Sigma : (y, \ldots, y) \in \T \mapsto y \quad (y \in S^{\nvar}).
    \end{equation} 
(See subection \ref{SS:calibration}.) So $\Sigma$ maps $\T$ homeomorphically 
onto $\F$. 

Let $\Phi : \D \partlyto \F := S^{\nvar}$ and consider the condition
	\begin{equation} \label{E:Phi.on.diagnl.exact}
	    \text{For every } y \in S^{\nvar}, \Phi(y, \ldots, y) 
	      \text{ is defined and  equals } y.
	\end{equation}
A $\Phi$ satisfying \eqref{E:Phi.on.diagnl.exact} is sensibly calibrated in the sense of section \ref{R:regularization.generalities}. If $\Phi$ satisfies \eqref{E:Phi.on.diagnl.exact}, at least approximately and is defined at least almost everywhere in $\D$, call $\Phi$ a ``measure of location'' on $S^{\nvar}$. (For an example of what it means for \eqref{E:Phi.on.diagnl.exact} to hold only approximately see \eqref{E:location.Phi.on.diagnl.loose} below. Sometimes, as in corollary \ref{C:local.exactness.of.fit} below, we relax the assumption that $\Phi$ be defined almost everywhere on $\D$.) Let $\Phi$ be such a measure of location and let $\Ss$ be the singular set of $\Phi$. Thus, $\Phi : \D \setminus \Ss \to S^{\nvar}$ is continuous. (So there are three $S$'s: The sphere $S^{\nvar}$, the group $S_{n}$, and the singular set $\Ss$.) We remind the reader of \eqref{E:D'.=.D.less.S}. 

Without further assumptions $\Phi$ may have no singularities. I.e., $\Ss$ may be empty.
For example, consider the ``measure of location'', 
$\Phi_{proj} : y_{1}, \ldots, y_{n} \mapsto y_{1}$, i.e., projection onto the first argument. $\Phi$ is continuous. The only $r = 0, \ldots, \dim \T = \nvar$ for which \eqref{E:nontriv.r-dim.homol} holds are $r = 0$ and $\nvar$. For those values of $r$ and $\Ss' = \varnothing$, 
$\Phi$ satisfies theorem \ref{T:Phi.star.Hr.contains.Theta.star.Hr}. But $\Phi$ does not satisfy proposition \ref{P:sing.dim.when.H.d-r.D.=.0} because $\check{H}^{d-r}(\D)$ is non-trivial for $r = 0, \nvar$. 

However, it is customary to assume that a measure of location is  symmetric in its arguments, i.e. that $\Phi \bigl( y_{\sigma(1)}, \ldots, y_{\sigma(n)} \bigr)$ is constant in permutations 
$\sigma$ of $\{ 1, \ldots, n \}$. We will follow that practice. 
$\Phi_{proj}$ does not satisfy this condition. In particular, this symmetry condition implies
    \begin{equation}  \label{E:S.is.perm.invar}
      \Ss \text{ is invariant under permutation of the $n$ factors of } \D. 
    \end{equation}
The paper \cite{spE91.top.direct.axis} applies the results in Eckmann \emph{et al} \cite{bEtGpjH62.genMeans} to show that any symmetric measure of location 
on $S^{\nvar}$ must have singularities. 
(The same issue arises, in a different guise is social choice theory. See, e.g., Chichilnisky  \cite{gC79.SocialChoiceTopol} and Ghrist \cite[Section 8.7, pp.\ 170--171]{rG14.ElemAppliedTopol}.) Here, we apply the results in chapter \ref{Chptr:topology} to compute a lower bound on the dimension of the singular set of a general measure of location on a sphere. (In chapter \ref{Chptr:linear.classification} we apply the result to the problem of linear classification on a sphere.)
Recall ``$\check{H}^{\ast}$'' is the \v{C}ech cohomology functor (Dold \cite[Chapter VIII, chapter 6]{aD95.alg.topol}, Munkres \cite[\S 73]{jrM84}). We use integer, i.e.\ 
$\mathbb{Z}$, coefficients for (co)homology. 

   \begin{theorem}   \label{T:spher.loc.sing.codim.nvar+1}
   Suppose $n > 1$ and $\nvar > 0$ and define $\D$ and $\T$ as above. 
   Let $\Ss' \subset \D$  be invariant under permutation of the $n$ factors of $\D$. 
   Let $\Phi : \D \setminus \Ss' \to \F := S^{\nvar}$ be continuous.
   Assume \eqref{E:Phi.on.diagnl.exact} holds. Assume
         \begin{equation} \label{E:loc.measure.has.no.sings.near.T}
           \Ss' \text{ is closed, has empty interior, and } \Ss' \cap \T = \varnothing.
         \end{equation}
      Assume
         \begin{multline} \label{E:assume.Phi.S'.sym}
             \text{$\Phi$ is symmetric in its arguments and } \Ss'
               \text{ is invariant } \\
                 \text{ under permutation of the factors }
                   S^{\nvar} \times \cdots \times S^{\nvar} \text{ of } \D .
         \end{multline}
      Then 
               \begin{equation}   \label{E:S.has.nontriv.cohom}
			\text{For some $k = 0, \ldots, \nvar$, we have } 
			  \check{H}^{n \nvar - k-1}(\Ss') \ne \{0\},  
	      \end{equation}
      Therefore, by \eqref{E:cohom.and.codim}, 
		\begin{equation}  \label{E:S.has.poz.n.nvar-nvar-1.measure}
			\Hm^{n \nvar - \nvar-1}(\Ss') > 0.
		\end{equation} 
	In particular, 
         \begin{equation}    \label{E:bound.on.sing.set.dim.in.sphere.loc}
            \text{codim} \, \Ss' \leq \nvar + 1.
         \end{equation}
   \end{theorem}

Hypothesis \eqref{E:loc.measure.has.no.sings.near.T} implies that \textbf{hypothesis \ref{Hyp:S.cap.T.small}} of Theorem \ref{T:Phi.star.Hr.contains.Theta.star.Hr} holds. Below, \eqref{E:loc.meas.has.no.severe.sings.near.T}, we relax hypothesis \eqref{E:loc.measure.has.no.sings.near.T}. We relax hypothesis \eqref{E:Phi.on.diagnl.exact}, as well. See \eqref{E:location.Phi.on.diagnl.loose}.

  \begin{proof}[Proof of Theorem \ref{T:spher.loc.sing.codim.nvar+1}]
$\Ss'$ is closed by assumption. We prove the following. Recall that  $\T \subset \D$ is the diagonal
	\begin{equation}   \label{E:sphere.loc.test.pattern.space}
	      \T = \{ (y_{1}, \ldots, y_{n}) \in \D : y_{1} = \cdots = y_{n} \in S^{\nvar} \}.
	\end{equation}

Assume \eqref{E:Phi.on.diagnl.exact} holds. Thus, 
$\Theta := \Phi \restriction_{\T} : \T \to S^{\nvar}$ is a homeomorphism. Hence, 
$\Theta_{\ast} : H_{\nvar}(\T) \to H_{\nvar}(S^{\nvar}) \isomto \mathbb{Z}$ is nontrivial.
By hypothesis \eqref{E:loc.measure.has.no.sings.near.T} we have  
   \begin{equation} \label{E:S.and.T.disjoint}
      \Ss' \cap \T = \varnothing.
   \end{equation}
Thus, theorem \ref{T:Phi.star.Hr.contains.Theta.star.Hr} applies in this case and \eqref{E:nontriv.r-dim.homol} holds with $r = \nvar$. Therefore, the homomorphism 
    \begin{equation}  \label{E:Phi.star.to.H.nvar.sphere.nontriv}
      \Phi_{\ast} : H_{\nvar}(\D \setminus \Ss') \to H_{\nvar}(S^{\nvar}) 
        \text{ is nontrivial. }
    \end{equation} 
But by Munkres \cite[Example 2, p.\ 346 and Theorem 53.1, p.\ 320]{jrM84} and Dold \cite[Chapter VIII, Propositions 1.3, p.\ 248 and 6.12, p.\ 285]{aD95.alg.topol}, we have $\check{H}^{d-r}(\D) \neq \{ 0 \}$. Therefore, proposition \ref{P:sing.dim.when.H.d-r.D.=.0} does not apply so one still cannot get 
\eqref{E:S.has.nontriv.cohom}  
without further work.
 
Suppose \eqref{E:S.has.nontriv.cohom} is false. I.e., \emph{suppose}
   \begin{equation} \label{E:cohom.S.vanishes.in.large.dim}
      \check{H}^{n \nvar - k-1}(\Ss') = 0 \quad \forall \; k = 0, \ldots, \nvar.
   \end{equation}
Let 
	\begin{equation*}
		i : \D \hookrightarrow (\D, \Ss') \text{ and } 
		  j : \D \setminus \Ss' \hookrightarrow \D \text{ be inclusions.}
	\end{equation*}
By \eqref{E:cohom.S.vanishes.in.large.dim} 
(and Dold \cite[Proposition 6.10, p.\ 284]{aD95.alg.topol}), for $\ell \leq \nvar$
   \[
      0 = \check{H}^{n \nvar-\ell}(\Ss') \leftarrow \check{H}^{n \nvar-\ell}(\D) 
         \xleftarrow{\check{i}} \check{H}^{n \nvar-\ell}(\D, \Ss') 
           \leftarrow \check{H}^{n \nvar-\ell-1}(\Ss') = 0,
   \]
where the equalities on the left and right follow from \eqref{E:cohom.S.vanishes.in.large.dim} with $k = \ell-1$ and $k = \ell$, resp. Thus, 
$\check{i}$ is an isomorphism in dimensions $n \nvar - \ell$ for $0 \leq \ell \leq \nvar$. Let $\ell \leq \nvar$ be an arbitrary non-negative integer. 
By Poincar\'{e}-Lefschetz duality (Dold \cite[Proposition 7.2, p.\ 292 and (7.6), p.\ 293]{aD95.alg.topol}), 
there is a commutative diagram\footnote{The commutativity of that square follows from Dold \cite[(7.6), p.\ 293]{aD95.alg.topol}. Here are the details. 
Take $M := \D$, $(\tilde{K}, \tilde{L}) := (\D, \varnothing)$, and $(K,L) := (\D, \Ss')$. 
Dold \cite[(7.6), p.\ 293]{aD95.alg.topol} is actually a little ambiguous. In its first appearance in the formula, $i'$ is the inclusion map 
$(M \setminus L, M \setminus K) \hookrightarrow (M \setminus \tilde{L}, M \setminus \tilde{K})$. Call this $i'_{L}$. Thus, $i'_{L}$ is the inclusion $j$ we already defined. The inclusion $i$ in Dold \cite[(7.6), p.\ 293]{aD95.alg.topol} is just the inclusion 
$i : M = \D \hookrightarrow (\D, \Ss') = (K,L)$ that we already defined. 

In its second appearance, $i'$ is the inclusion $(M, M \setminus K) \hookrightarrow (M, M \setminus \tilde{K})$. Call this $i'_{R}$. Thus, $i'_{R} : \D \to \D$ is the identity. $\xi$ is the orientation class, $o \in H_{n \nvar}(\D)$ (see proof of theorem \ref{T:Phi.star.Hr.contains.Theta.star.Hr}). The vertical isomorphisms in \eqref{E:i.j.S.D.com.diag}  just correspond to cap product with $o$. Let $x \in \check{H}^{n \nvar-k}(K,L) = \check{H}^{n \nvar-k}(\D, \Ss')$. Then Dold \cite[(7.6), p.\ 293]{aD95.alg.topol} becomes
	\[
		j_{\ast}(x \cap o) = i'_{L, \ast}(x \cap \xi) = \check{i} (x) \cap i'_{R, \ast} (\xi) 
		             = \check{i} (x) \cap o,
	\]
which is exactly the diagram \eqref{E:i.j.S.D.com.diag}.} 
   \begin{equation}    \label{E:i.j.S.D.com.diag}
      \begin{CD}
         H_{\ell}(\D \setminus \Ss') @>j_{\ast}>> H_{\ell}(\D)  \\
         @A\isomto AA          @AA\isomto A  \\
         \check{H}^{n \nvar-\ell}(\D, \Ss') @>>\check{i} \; (\isomto) > \check{H}^{n \nvar-\ell}(\D).
      \end{CD}
   \end{equation}
Therefore, 
   \begin{equation}  \label{E:j*.is.isom.in.low.dim}
      j_{\ast} \text{ is an isomorphism in dimensions $\nvar$ or lower.}
   \end{equation}
By Munkres \cite[Example 2, p.\ 346]{jrM84}, 
	\begin{equation}  \label{E:H.p-1.is.triv.or.cyc}
		H_{\nvar-1}(\D) \text{ is trivial, if $\nvar > 1$, or infinite cyclic if } \nvar = 1.
	\end{equation}
The same thing goes for $S^{\nvar}$.
By \eqref{E:j*.is.isom.in.low.dim} the same thing goes for 
$H_{\nvar-1}(\D \setminus \Ss')$. Therefore,
by Munkres \cite[Theorem 52.3(b), p.\ 318]{jrM84}
   \begin{equation}  \label{E:Ext.H.p-1.D.less.S.is.0}
       \text{Ext}[H_{\nvar-1}(S^{\nvar}), \mathbb{Z}] 
         = \text{Ext}[H_{\nvar-1}(\D), \mathbb{Z}] = 
          \text{Ext}[H_{\nvar-1}(\D \setminus \Ss'), \mathbb{Z}] = 0. 
   \end{equation}

Let ``id'' generically denote the identity map or homorphism. By Munkres \cite[Theorem 41.1(a), p.247]{jrM84} \eqref{E:j*.is.isom.in.low.dim}, implies
    \begin{equation}  \label{E:Hom.j*.is.isom}
       \text{Hom}(j_{\ast}, \text{id}) : \text{Hom}[H_{\nvar}(\D), \mathbb{Z}]
          \to \text{Hom}[H_{\nvar}(\D \setminus \Ss'), \mathbb{Z}]
          \text{ is an isomorphism.}
    \end{equation}

By \eqref{E:Phi.star.to.H.nvar.sphere.nontriv}, if $f : H_{\nvar}(S^{\nvar}) \to \mathbb{Z}$ is an isomorphism then $f \circ \Phi_{\ast} : H_{\nvar}(\D \setminus \Ss') \to \mathbb{Z}$ is nontrivial. 
Thus, $\text{Hom}(\Phi_{\ast}, \text{id}) : \text{Hom}[H_{\nvar}(S^{\nvar}), \mathbb{Z}]  \to \text{Hom}[H_{\nvar}(\D \setminus \Ss'), \mathbb{Z}]$ is nontrivial. 
(See the definition in Munkres \cite[p.\ 248]{jrM84}.)  
By Munkres \cite[Corollary 53.2 p.\ 323]{jrM84} and
\eqref{E:Ext.H.p-1.D.less.S.is.0} we have the following homomorphism of exact sequences.
 \small
  \[
   \begin{CD}
    & & & & & & 0        \\
    & & & & & & @|       \\
    0 @<<< \text{Hom}[H_{\nvar}(S^{\nvar}), \mathbb{Z}]  @<\isomto<<  H^{\nvar}(S^{\nvar})  @<<<  
                                                                        \text{Ext}[H_{\nvar-1}(S^{\nvar}), \mathbb{Z}]  @<<< 0 \\ 
    & &    @VV\text{Hom}(\Phi_{\ast}, \text{id})V    @VV\Phi^{\ast}V    @V\text{Ext}(\Phi_{\ast}, \text{id})VV  \\ 
    0 @<<< \text{Hom}[H_{\nvar}(\D \setminus \Ss'), \mathbb{Z}]  @<\isomto<<  H^{\nvar}(\D \setminus \Ss')  
                                                         @<<<  \text{Ext}[H_{\nvar-1}(\D \setminus \Ss'), \mathbb{Z}]  
                                                                        @<<< 0  \\
   & & & & & & @|   \\
   & & & & & & 0,
   \end{CD}
  \]
 \normalsize 
where ``$H^{\ast}$'' indicates singular cohomology. It follows that
   \begin{equation}   \label{E:lambda^*.is.nontriv.in.dim.p}
      \Phi^{\ast} : H^{\nvar}(S^{\nvar}) \to H^{\nvar}(\D \setminus \Ss') 
         \text{ is nontrivial.}
   \end{equation}

Similarly, by \eqref{E:Hom.j*.is.isom}, we have,
 \small
   \[
   \begin{CD}    & & & & & & 0        \\
    & & & & & & @|       \\
    0 @<<< \text{Hom}[H_{\nvar}(\D), \mathbb{Z}]  @<\isomto<<  H^{\nvar}(\D)  
        @<<<  \text{Ext}[H_{\nvar-1}(\D), \mathbb{Z}]  @<<< 0 \\ 
    & &    @VV\text{Hom}(j_{\ast}, \text{id})\;(\isomto)V    @VVj^{\ast}V    
        @V\text{Ext}(j_{\ast}, \text{id})VV  \\ 
    0 @<<< \text{Hom}[H_{\nvar}(\D \setminus \Ss'), \mathbb{Z}]  @<\isomto<<  
        H^{\nvar}(\D \setminus \Ss') @<<<  \
            text{Ext}[H_{\nvar-1}(\D \setminus \Ss'), \mathbb{Z}]  @<<< 0.  \\
   & & & & & & @|   \\
   & & & & & & 0
   \end{CD}
   \]
 \normalsize
It follows from \eqref{E:j*.is.isom.in.low.dim} that 
   \begin{equation}  \label{E:j^*.is.isom.in.dim.p}
      j^{\ast} : H^{\nvar}(\D) \to H^{\nvar}(\D \setminus \Ss') \text{ is an isomorphism.}
   \end{equation}

Let $\alpha \in H^{\nvar}(S^{\nvar})$ be a generator of $H^{\nvar}(S^{\nvar})$. By \eqref{E:H.p-1.is.triv.or.cyc} and Munkres 
\cite[Theorem 60.5, p.\ 358; Theorem 54.4(c), p.\ 329; and Theorem 50.8, p.\ 305]{jrM84}, the following classes consitute a basis for $H^{\nvar}(\D)$
   \begin{align}  \label{E:bk.=.1sxa}
      \text{\small{$k^{th}$}} & \text{\small{factor}}  \notag \\
	      &\downarrow  \\
	         \beta^{k} := 1 \times \cdots \times 1 \times & \,\alpha \times 1 \times \cdots \times 1 
	            \quad (k = 1, \ldots, n),     \notag
   \end{align}
where ``$\times$'' denotes the
cohomology cross product and $1 \in H^{0}(S^{\nvar})$ is the unit element. It follows from \eqref{E:j^*.is.isom.in.dim.p} that 
	\begin{equation}  \label{E:bks.form.basis}
		j^{\ast}(\beta^{1}), \ldots, j^{\ast}(\beta^{n}) \text{ form a basis of } 
		  H^{\nvar}(\D \setminus \Ss').
	\end{equation}

Let $\tau$ be a permutation of $(1, \ldots, n)$. Define 
$\tilde{\tau} : \D \to \D$ to be the map that performs the corresponding permutation of coordinates on $\D$. I.e., if $y_{1}, \ldots, y_{n} \in S^{\nvar}$ then 
$\tilde{\tau}(y_{1}, \ldots, y_{n}) = (y_{\tau(1)}, \ldots, y_{\tau(n)})$. By \eqref{E:assume.Phi.S'.sym} $\Ss'$ is invariant under permutation of coordinates. Therefore, $\tilde{\tau}$ maps $\D \setminus \Ss'$ into itself. Also denote by $\tilde{\tau}$ the restriction of $\tilde{\tau}$ to $\D \setminus \Ss'$ regarded as a map 
into $\D \setminus \Ss'$. The operation $\tau \mapsto \tilde{\tau}$ is functorial: 
$\widetilde{ \tau_{1} \circ \tau_{2} } = \tilde{\tau}_{1} \circ \tilde{\tau}_{2}$.
We have
   \begin{equation}  \label{E:j.and.tau.commute}
      j \circ \tilde{\tau} = \tilde{\tau} \circ j \text{ on } \D \setminus \Ss'.
   \end{equation}
Since $\Phi$ is symmetric in its arguments (\eqref{E:assume.Phi.S'.sym}), 
   \begin{equation}  \label{E:Phi.sigma.tild.=.lambda}
      \Phi \circ \tilde{\tau} = \Phi \text{ on } \D \setminus \Ss'.
   \end{equation}
By Munkres\cite[Theorem 61.2, p.\ 361]{jrM84}\footnote{
If $\tau: \{ 1, 2, \ldots, n \} \to \{ 1, 2, \ldots, n \}$ is a simple transposition, i.e., permutation that just swaps two adjacent numbers, then it is immediate from
Munkres \cite[Theorem 61.2, p.\ 361]{jrM84} and the naturality of the cohomology cross product 
(Munkres \cite[Theorem 60.5, p.\ 358]{jrM84}) that
$\tilde{\tau}^{\ast}(\beta^{k}) = \beta^{\tau(k)} = \beta^{\tau^{-1}(k)}$. If $\tau$ is any permutation then it can be written as a product $\tau = \tau_{1} \circ \cdots \circ \tau_{q}$ of simple transpositions. Thus, since the cohomology functor is contravariant,
$\tilde{\tau}^{\ast}(\beta^{k}) = \beta^{\tau_{q}^{-1} \circ \cdots \circ \tau_{1}^{-1}(k)}$. 
But $\tau_{q}^{-1} \circ \cdots \circ \tau_{1}^{-1} = \tau^{-1}$.
}, 
   \begin{equation} \label{E:perm.form}
      \tilde{\tau}^{\ast}(\beta^{k}) = \beta^{\tau^{-1}(k)} \quad (k = 1, \ldots, n).
   \end{equation}

Let 
    \begin{equation}  \label{E:gamma:=.Phi.star(alpha)}
      \gamma = \Phi^{\ast}(\alpha) \in H^{\nvar}(\D \setminus \Ss') .
    \end{equation}
Then, by \eqref{E:lambda^*.is.nontriv.in.dim.p}, $\gamma \ne 0$. By \eqref{E:bks.form.basis},
we may write
   \[
      \Phi^{\ast}(\alpha) = \gamma = \sum_{k = 1}^{n} m_{k} \, j^{\ast}(\beta^{k}),
   \]
where $m_{1}, \ldots, m_{n} \in \mathbb{Z}$ are not all 0. Applying \eqref{E:Phi.sigma.tild.=.lambda}, \eqref{E:j.and.tau.commute}, and \eqref{E:perm.form} to this we get,
   \[
      \sum_{k = 1}^{n} m_{k} \, j^{\ast}(\beta^{k}) = \gamma = \tilde{\tau}^{\ast} \circ \Phi^{\ast}(\alpha) 
         = \sum_{\ell = 1}^{n} m_{\tau(\ell)} \, j^{\ast}(\beta^{\ell}) \quad 
              \text{ for every permutation } \tau.
   \]
Since $j^{\ast}(\beta^{1}), \ldots, j^{\ast}(\beta^{n})$ are linearly independent, it follows that 
$m_{1} = \cdots = m_{n} = m \ne 0$ for some $m \in \mathbb{Z}$. I.e.,
   \begin{equation} \label{E:gamma.is.mult.of.sum}
      \gamma = m \sum_{k = 1}^{n} j^{\ast}(\beta^{k}).
   \end{equation}

Let $\Delta$ be the diagonal map, 
    \begin{equation}  \label{E:Delta.is.diag.map.on.sphere}
      \Delta(y) = (y, \ldots, y) \in \D \; (y \in S^{\nvar}) .
    \end{equation}  
By \eqref{E:loc.measure.has.no.sings.near.T}, 
$\Delta(S^{\nvar}) \subset \D \setminus \Ss'$. Thus, 
$\Delta$ may be thought of as having codomain $\D$ or $\D \setminus \Ss'$. With this understanding, we have
   \begin{equation} \label{E:j.connects.2.deltas}
      \Delta = j \circ \Delta.
   \end{equation}
By \eqref{E:Phi.on.diagnl.exact},
   \begin{equation} \label{E:Delta.is.rite.inverse.of.Phi}
      (\Phi \restriction_{\T}) \circ \Delta = \Phi \circ \Delta = \text{identity on } S^{\nvar}.
   \end{equation}
Thus, by \eqref{E:Delta.is.rite.inverse.of.Phi}, \eqref{E:gamma:=.Phi.star(alpha)}, \eqref{E:gamma.is.mult.of.sum}, \eqref{E:j.connects.2.deltas}, \eqref{E:bk.=.1sxa}, 
Munkres \cite[Theorem 61.3, p.\ 362]{jrM84}, and the fact that, by assumption, $n > 1$, we have
    \begin{equation}  \label{E:alpha=mn.alpha}
        \alpha = \Delta^{\ast} \circ \Phi^{\ast}(\alpha) = \Delta^{\ast}(\gamma)
          = m \sum_{k = 1}^{n} \Delta^{\ast} \circ j^{\ast}(\beta^{k}) 
             = m \sum_{k = 1}^{n} \Delta^{\ast}(\beta^{k})
                 = mn \: \alpha \ne \alpha,
    \end{equation}
This contradiction means \eqref{E:cohom.S.vanishes.in.large.dim} must be false, i.e., \eqref{E:S.has.nontriv.cohom} holds.
  \end{proof}

  \begin{corly}  \label{C:homotopic.to.identity.on.sphere}
Let $\Delta$ be the diagonal map, $\Delta(y) = (y, \ldots, y) \in \D$ 
($y \in S^{\nvar}$). Suppose all the hypotheses of theorem \ref{T:spher.loc.sing.codim.nvar+1} hold except instead of \eqref{E:Phi.on.diagnl.exact} only require that 
    \begin{equation} \label{E:Phi.circ.Delta.homot.to.identity.on.sphere}
      \Phi \circ \Delta \text{ is homotopic to the identity on } S^{\nvar}.
    \end{equation} 
Then \eqref{E:S.has.nontriv.cohom}, \eqref{E:S.has.poz.n.nvar-nvar-1.measure}, and \eqref{E:bound.on.sing.set.dim.in.sphere.loc} all hold.
  \end{corly}
  \begin{proof}
Let $\Theta := \Phi \restriction_{\T} : \T \to S^{\nvar}$. Under the weakened assumption it remains the case that 
$\Theta_{\ast} : H_{\nvar}(\T) \to H_{\nvar}(S^{\nvar}) \isomto \mathbb{Z}$ remains nontrivial and \eqref{E:alpha=mn.alpha} still holds. So the proof of the theorem still goes through under the weaker hypostheses.
  \end{proof}
For an example of such a homotopy see remark \ref{R:loose.exactness.of.fit.and.homotopy}.

  \begin{remark}[Retraction of neighborhood of $\T$ for directional location problem] \label{R:retraction.sphere.loc} 
In order to apply theorem \ref{T:if.lin.combo.on.F.then.can.rstrct.to.bad.sings} part \ref{I:Omega.Phi.agree.on.Pf} we need a closed set, $\Pf$, of perfect fits s.t.\ $\Pf \cap \D'$ is dense in $\Pf$; a neighborhood, $\Rcl \subset \D$, of $\Pf$; and a retraction, $r : \Rcl \to \Pf$. Take $\Pf := \T$ defined in \eqref{E:directional.T.defn}. Then $\Pf$ is closed and, by \eqref{E:loc.measure.has.no.sings.near.T} and \eqref{E:D'.=.D.less.S}, $\Pf \cap \D' = \Pf$. The neighborhood $\Rcl$ needs to be $G$-invariant and $r$ needs to be $G$-equivariant, i.e., $r \circ g = g \circ r$ for every $g \in G$. Here, $G$ is a sensible finite group acting on $\D$. From \eqref{E:assume.Phi.S'.sym}, we see that the natural choice of $G$ is the permutation group on $n$ symbols.
Let $y \cdot y'$ be the usual ``dot'' product of $y, y' \in \RR^{\nvar +1}$. Take $\Rcl \supset \Pf$ to be the set of $x = (y_{1}, \ldots, y_{n}) \in (S^{\nvar})^{n}$ s.t.\ for every $i, j = 1, \ldots, n$ we have $y_{i} \cdot y_{j} > 0$. Thus, if $x \in \Rcl$ 
we have $| y_{1} + \cdots + y_{n} | > 0$. 
If $x \in \Rcl$ let $y := | y_{1} + \cdots + y_{n} |^{-1} ( y_{1} + \cdots + y_{n} ) \in S^{\nvar}$ and 
define $r(x) = (y, \ldots y) \in \Pf = \T \subset \D$. Obviously, $r \circ g = g \circ r$.
(See remark \ref{R:retraction.in.manifs}.)
  \end{remark}

\section{Severe singularities for location problem on $S^{\nvar}$}  \label{S:convex.combos.on.spheres}
It is easy to to construct a convex combination function (definition \ref{D:convex.combo.fn}) on $S^{\nvar}$. Recall that $\F = S^{\nvar}$. If $u \in \F$ and $\theta \in (0, \pi/2]$, denote by $V_{u, \theta}$ the spherical cap
	\begin{equation}  \label{E:sphere.cap.defn}
	      V_{u, \theta} = \{ v \in \F : v \cdot u > \cos \theta \}.
	 \end{equation}
Thus, $V_{u, \theta}$ is a spherical cap of radius $\theta$ and $V_{u,\pi/2}$ ($u \in \F$) is an open hemisphere. Let
   \begin{equation}    \label{E:V.theta.defn}
      \msf{V}_{\theta} = \{ V_{u, \theta} \subset \F : u \in \F \}.
   \end{equation}
This cover is similar to the cover $\msf{V}_{\theta}$ defined in section \ref{SS:lin.combo.for.plane.fit.w/.k=nvar-1} except there $\F$ consists of lines (actually, orthogonal complements of lines) while  
here $\F$ consists of vectors. Thus a ``bad'' singularity (one in $\Ss^{\msf{V}_{\theta}}$) is a data set $x \in \D$ s.t.\ \emph{no} neighborhood 
of $x$ (in $\D'$) has an image under $\Phi$ whose closure lies in an open spherical cap (of radius $\theta$). (A subspace of $S^{q}$ that lies in no open hemisphere is ``taut''?)

It is easy to define a convex combination function (definition \ref{D:convex.combo.fn}) 
on $\msf{V}_{\theta}$. Let $u \in S^{\nvar}$; let $\theta \in (0, \pi/2]$; 
let $v_{0}, v_{1}, \ldots, v_{m} \in V_{u, \theta}$; 
and let $\lambda_{0}, \lambda_{1}, \ldots, \lambda_{m} \geq 0$  
with $\lambda_{0} + \lambda_{1} + \cdots + \lambda_{m} = 1$. 
Let $w := \lambda_{0} v_{0} + \cdots + \lambda_{m} v_{m}$. 
Then 
$u \cdot w > (\lambda_{0} + \lambda_{1} + \cdots + \lambda_{m}) \cos \theta 
= \cos \theta \geq 0$. In particular, $w \neq 0$. Let $V := V_{u, \theta}$ and define
   \begin{equation}   \label{E:lin.combo.on.sphere}
      \gamma \bigl( V, (\lambda_{0}, \lambda_{1}, 
        \ldots, \lambda_{m}), (v_{0}, v_{1}, \ldots, v_{m}) \bigr) = |w|^{-1} w. 
   \end{equation}
(See \cite{spE91.top.direct.axis}, section 2.) We also have $|w| \leq \lambda_{0} |v_{0}| + \cdots + \lambda_{m} |v_{m}| = 1$. Therefore, 
	\[
	      u \cdot 
	            \gamma \bigl( V, (\lambda_{0}, \lambda_{1}, 
        \ldots, \lambda_{m}), (v_{0}, v_{1}, \ldots, v_{m}) \bigr)  
		   \geq u \cdot w > (\lambda_{0} + \lambda_{1} + 
		            \cdots + \lambda_{m}) \cos \theta = \cos \theta.
	\]
I.e., $\gamma \bigl( V, (\lambda_{0}, \lambda_{1}, \ldots, \lambda_{m}), (v_{0}, v_{1}, \ldots, v_{m}) \bigr) \in V$, so property \ref{I:convex.combos.are.local} of definition \ref{D:convex.combo.fn} holds for $\gamma$.
Trivially, $\gamma$ satisfies properties \ref{I:consistency.of.conv.combos} through \ref{I:1.is.identity} of definition \ref{D:convex.combo.fn} as well as \eqref{E:drop.0.coefs}. We see that $\gamma$ is commutative (defined shortly after definition \ref{D:convex.combo.fn}).

Note that $\msf{V}_{\pi/2}$ means something very different from the $\msf{V}_{90^{\circ}}$ in section \ref{SS:lin.combo.for.plane.fit.w/.k=nvar-1} .
  
  \begin{remark} \label{R:convx.combo.fn.on.sphere}
The convex combination function defined in \eqref{E:lin.combo.on.sphere} 
with $\lambda_{0} = \cdots = \lambda_{m} = 1/(m+1)$ extends to a measure of location defined on 
	\[
		\bigl\{ (y_{0}, y_{1}, \ldots, y_{n}) \in \D : 
		  y_{0} + y_{1} + \cdots + y_{n} \neq 0 \bigr\}.
	\]
This measure of location, the ``directional mean'', is examined further in 
chapter \ref{Chptr:aug.direct.mean}.
  \end{remark}

By how much might $\Phi(x)$ swing as $x \in \D'$ varies near a $\msf{V}_{\pi/2}$-severe singularity? We argue as in the proof of proposition \ref{P:diam.oriented.90.degree.sing} to prove the following. Recall that, by \eqref{E:angle.between.vectors}, the codomain 
of $\arccos$ is $[0, \pi]$.

  \begin{prop}  \label{P:diam.of.image.of.nhbd.of.sing.of.loc.meas}
Let $\Phi$ be a measure of location on the sphere $S^{\nvar}$. Let 
$x \in \Ss^{\msf{V}_{\pi/2}}$. Let $\epsilon > 0$ and let $\clU$ be a neighborhood of $x$. Then there exists $x_{1}, x_{2} \in \clU \cap \D'$ s.t.\ 
$\arccos(-1/\nvar + \epsilon) < \angle \bigl( \Phi(x_{1}, \Phi(x_{2}) \bigr)$. 
In particular, for some $x_{1}, x_{2} \in \clU \cap \D'$, 
$\angle \bigl( \Phi(x_{1}), \Phi(x_{2}) \bigr)$ is obtuse.
  \end{prop}
Observe that if $\epsilon \in (0, 1/\nvar)$ then $\pi/2 < \arccos(-1/\nvar + \epsilon) < \pi$.
  \begin{proof} Let $\D'$ be the dense subset of $\D$ relative to which singularities of $\Phi$ are defined. 
Let $x \in \Ss^{\msf{V}_{\pi/2}}$. Let $\clU$ be a neighborhood of $x$ and let 
$\msf{A} := \overline{\Phi \bigl[ \clU \cap \D' \bigr]}$. Thus, $\msf{A}$ is contained in no $V_{u,\pi/2}$. Regard $\msf{A}$ as a subset of $\RR^{\nvar}$. $\msf{A}$ lies in the closed unit ball, $\overline{B_{1}(0)}$, centered at 0, obviously. (See \eqref{E:ball.defn}.) Let 
$\overline{B_{r}(y)}$, for some $y \in \RR^{\nvar+1}$, be a smallest closed ball (i.e. has smallest radius) containing 
$\msf{A}$. Thus, $\msf{A} \subset \overline{B_{r}(y)} \cap \overline{B_{1}(0)}$. The radius, $r$, of $\overline{B_{r}(y)}$ is $\leq 1$ and the center, $y$, does not have to be 0. 
If $r = 1$ then we might as well take $y = 0$. We will show $r = 1$. Assume $r < 1$. 
Since $\msf{A} \subset S^{\nvar}$, $y$ cannot be 0. Let $u := |y|^{-1} y \in S^{\nvar}$. 

\emph{Claim:} $\msf{A} \subset V_{u, \pi/2}$. For suppose not and let 
$w \in \msf{A} \setminus V_{u, \pi/2} \subset \overline{B_{r}(y)}$. 
Then $\cos \angle(w, u) = w \cdot u \leq 0$. (See \eqref{E:angle.between.vectors}.)
But $y \propto u$ so $\cos \angle(w, y) = \cos \angle(w, u) < 0$. Therefore, applying the Law of Cosines to the triangle with vertices $0, y, w$ we have
    \begin{equation*}
     r^{2} \geq |w-y|^{2} = |w|^{2} + |y|^{2} 
        - 2 |w| |y| \bigl( \cos \angle(w, y) \bigr) \geq |w|^{2} + |y|^{2} 
          = 1 + |y|^{2} > 1 > r^{2} .
    \end{equation*}
Contradiction. Therefore, $\msf{A} \subset V_{u, \pi/2}$,
which in turn contradicts the assumption that $x \in \Ss^{\msf{V}_{\pi/2}}$. Therefore, 
$r = 1$.

In other words, $\overline{B_{1}(0)}$ is the (a?)\ smallest ball containing $\msf{A}$. Therefore, by Jung's theorem (Jung \cite{hJ1901.Jungs.theorem}, \cite{hJ1910.Jungs.theorem}, Rademacher and Toeplitz \cite[Chapter 16]{hRoT1957.EnjoymentOfMath}, Wikipedia), $diam(\msf{A}) \geq \sqrt{2(\nvar+1)/\nvar} > \sqrt{2}$. (Here, ``$diam(\msf{A})$'' is Euclidean diameter. See \eqref{E:diam.of.set}.)

Let $\epsilon \in (0,1)$. Recall that $| \cdot |$ is Euclidean distance (in $\RR^{\nvar+1}$ in this case). 
Since $\msf{A} = \overline{\Phi \bigl[ \clU \cap \D' \bigr]}$ there exist 
$x_{1}, x_{2}  \in \clU \cap \D'$ s.t.\ 
$\bigl| \Phi(x_{1}) - \Phi(x_{2}) \bigr| > \sqrt{2(\nvar+1)/\nvar} - \epsilon/2$. 
Let $z_{i} := \Phi(x_{i})$ ($i=1,2$). Then
    \begin{equation*}
      2 - 2 z_{1} \cdot z_{2} = |z_{1}-z_{2}|^{2} > \frac{2(\nvar+1)}{\nvar} 
        - \epsilon \sqrt{\frac{2(\nvar+1)}{\nvar}} + \frac{\epsilon^{2}}{4} 
    \end{equation*}
Now, $\nvar \geq 1$, by \eqref{E:nvar.n.for.spheres}, and $2(\nvar+1)/\nvar$ decreases in $\nvar$. Therefore,  
    \begin{equation*}
         2 - 2 z_{1} \cdot z_{2} \geq \frac{2(\nvar+1)}{\nvar} + (-2 + 1/4) \epsilon 
           >  \frac{2(\nvar+1)}{\nvar} - 2 \epsilon .
    \end{equation*}
Hence, 
    \begin{equation*}
      \cos \angle(z_{1}, z_{2}) = z_{1} \cdot z_{2} < 1 - \frac{1(\nvar+1)}{\nvar} + \ \epsilon
        = -1/\nvar + \epsilon .
    \end{equation*}
Remembering that $\arccos$ is decreasing, we get 
$\angle(z_{1}, z_{2}) > \arccos (-1/\nvar + \epsilon)$, as desired.
  \end{proof}
Thus, for measures of location on the circle ($\nvar = 1$) arbitrarily close to a 
$\msf{V}_{\pi/2}$-severe singularity one can find $x_{1}, x_{2} \in \D'$ s.t.\ 
$\angle \bigl( \Phi(x_{1}), \Phi(x_{2}) \bigr) \bigr)$ is arbitrarily close to $\pi$ radians. This is achieved by the directional mean (remark \ref{R:convx.combo.fn.on.sphere}, chapter \ref{Chptr:aug.direct.mean}). Chapter \ref{Chptr:aug.direct.mean} shows that this generalizes. This is illustrated in figure \ref{F:CircleSing}. 

In corollary \ref{C:homotopic.to.identity.on.sphere} we relaxed the assumption \eqref{E:Phi.on.diagnl.exact}. We can use the ``severity trick'' (remark \ref{R:severity.trick} to relax assumption \eqref{E:loc.measure.has.no.sings.near.T} as well. Recall that $\Delta$ is the diagonal map, 
$\Delta(y) = (y, \ldots, y) \in \D$.
  \begin{corly}  \label{C:Haus.measure.of.severe.sings.on.spheres}
Suppose $\Phi : \D \partlyto S^{q}$, continuous on $\D \setminus \Ss$, satisfies \eqref{E:D'.=.D.less.S} and \eqref{E:assume.Phi.S'.sym}.
In place of \eqref{E:Phi.on.diagnl.exact} assume the restriction, 
$\Phi \restriction_{\T \setminus \Ss}$, of $\Phi$ to $\T \setminus \Ss$ has a unique continuous extension $\Theta$ to all of $\T$ (i.e., hypothesis \ref{Hyp:extend} of theorem \ref{T:Phi.star.Hr.contains.Theta.star.Hr} holds) and 
    \begin{equation} \label{E:Theta.circ.Delta.homot.to.id.on.sphere}
      \Theta \circ \Delta \text{ is homotopic to the identity on } S^{\nvar} .
    \end{equation} 
In place of \eqref{E:loc.measure.has.no.sings.near.T} assume  
      \begin{equation}   \label{E:loc.meas.has.no.severe.sings.near.T}
          \T \setminus \Ss \text{ is dense in } \T
            \text{ and } \Ss^{\msf{V}_{\pi/2}} \cap \T = \varnothing .
      \end{equation}
Then there is a data map $\mu' : \D \partlyto S^{q}$, continuous 
on $\D \setminus \Ss^{\msf{V}_{\pi/2}}$, s.t.\ the restriction $\mu' \restriction_{\T}$ 
equals $\Theta$. Thus, 
\eqref{E:Phi.circ.Delta.homot.to.identity.on.sphere} holds with $\mu'$ in place of $\Phi$.
Moreover, \eqref{E:loc.measure.has.no.sings.near.T}, 
and \eqref{E:assume.Phi.S'.sym} hold 
with $\mu'$ in place of $\Phi$ and $\Ss^{\msf{V}_{\pi/2}}$ in place of $\Ss'$.

We have $R := dist_{n \nvar - \nvar -1}(\Ss^{\msf{V}_{\pi/2}}, \T) > 0$ and, 
even if $\Ss$ is not closed,
    \begin{equation}  \label{E:Haus.measure.severe.sings.on.spheres}
      \Hm^{n \nvar - \nvar -1}(\Ss) \geq \Hm^{n \nvar - \nvar -1}(\Ss^{\msf{V}_{\pi/2}}) 
        \geq \gamma R^{n \nvar - \nvar -1} > 0
    \end{equation}
for some $\gamma > 0$ not depending on $\Phi$ or $\Ss$.  
(If $n \nvar - \nvar -1 = 0$, then, by \eqref{E:0.dim.Haus.measure} and \eqref{E:Haus.measure.severe.sings.on.spheres}, 
$\Ss^{\msf{V}_{\pi/2}} \neq \varnothing$.) 
In particular, $\text{codim} \, \Ss \leq \text{codim} \, \Ss^{\msf{V}_{\pi/2}} \leq \nvar + 1$.
  \end{corly}

Obviously, \eqref{E:loc.measure.has.no.sings.near.T} implies \eqref{E:loc.meas.has.no.severe.sings.near.T}. In \eqref{E:loc.meas.has.no.severe.sings.near.T}, we allow singularities in $\T$, just so long as they are not $\msf{V}_{\pi/2}$-severe. Otherwise, we do not restrict the size of the set of singularities in $\T$. This is in the spirit of the idea enunciated at the beginning of section \ref{SS:approx.cont.}. 
  \begin{proof}
The existence of $\mu'$ follows from theorem \ref{T:if.lin.combo.on.F.then.can.rstrct.to.bad.sings} and remark \ref{R:retraction.sphere.loc}. Therefore, by corollary \ref{C:homotopic.to.identity.on.sphere}, \eqref{E:S.has.nontriv.cohom}, \eqref{E:S.has.poz.n.nvar-nvar-1.measure}, and \eqref{E:bound.on.sing.set.dim.in.sphere.loc} all hold with $\mu'$ in place of $\Phi$ and 
$\Ss^{\msf{V}_{\pi/2}}$ in place of $\Ss'$. 
That $R > 0$ is immediate from \eqref{E:loc.meas.has.no.severe.sings.near.T}, \eqref{E:SV.is.closed}, and compactness of $\D$. 

Let $a := n \nvar - \nvar -1$. \emph{Claim:} 
$(\mu', \Ss^{\msf{V}_{\pi/2}}, S_{n}, \T, a)$, where as usual 
$S_{n}$ is the symmetric group on $n$ symbols (see \eqref{E:symbol.for.symmetric.group}) acting on $\D$, has property \ref{Pty:agree.near.T}. 
That part \ref{I:Phi.is.nice} of the property holds is immediate. 

Let $\Psi$ be as in part \ref{E:Psi.is.nice} of the property, so $\tilde{\Ss}$ is closed with empty interior and \eqref{E:assume.Phi.S'.sym} holds 
for $\Psi$. 
Then $\tilde{\Ss} \cap \T = \Ss^{\msf{V}_{\pi/2}} \cap \T = \varnothing$ and 
\eqref{E:Phi.circ.Delta.homot.to.identity.on.sphere} holds with $\Psi$ in place of $\Phi$, because it holds with $\mu'$ in place of $\Phi$. Therefore, by corollary \ref{C:homotopic.to.identity.on.sphere} again, 
\eqref{E:S.has.poz.n.nvar-nvar-1.measure} holds for $\Psi$, so $\Hm^{a}(\tilde{\Ss}) > 0$. Thus, part \ref{E:Psi.is.nice} of property \ref{Pty:agree.near.T} holds, too. 
So the claim that $(\mu', \Ss^{\msf{V}_{\pi/2}}, G, \T, a)$ has property \ref{Pty:agree.near.T} is proved. 

Since $\Pf = \T$ is a manifold, by example \ref{Ex:TubeNbhdSpecialCaseOfCones}, it has a neighborhood in $T \D \restriction_{\Pf}$ fibered over $\T$ by cones. And, by \eqref{E:Riem.metric.on.prod.of.spheres.is.invar}, the Riemannian metric on $\D$ is invariant under $S_{n}$.  See also \eqref{E:prod.of.spheres.has.perm.invar.triangulation}.  Therefore, we may apply theorem \ref{T:lwr.bnd.on.Haus.meas}. The ``$\Pf$'' in theorem \ref{T:lwr.bnd.on.Haus.meas} is $\T = S^{\nvar}$ so $p = \nvar$ and $d = n \nvar$. Therefore, $d-p-1 = n \nvar - \nvar - 1$. The ``$a$'' in \eqref{E:Hm.a.S.geq.R.d-p-1}  also equals $n \nvar - \nvar - 1$ and \eqref{E:Haus.measure.severe.sings.on.spheres} holds.
(By \eqref{E.rho.plays.role.of.xi.in.loc.on.spheres}, $\Hm^{n \nvar - \nvar - 1}$ is computed using $\rho$ defined in \eqref{E:rho.is.top.metric}.)
  \end{proof}. 
  
  \begin{example}  \label{E:really.dumb.Phi}
Let $n=3$ and $\nvar = 1$. Define $\Psi : \D \partlyto S^{1}$ as follows. On $\T$, 
$\Psi$ sends $(y,y,y) = \Delta(y)$ to $y \in S^{1}$. If $x \in \D \setminus \T$ 
define $\Psi(x) = (-1,0) \in S^{1}$. Every point of $\T$ except 
$\Delta \bigl( (-1,0) \bigr)$ is a singularity of $\Psi$ but $\Delta \bigl( (1,0) \bigr)$ is the only $\msf{V}_{\pi/2}$-severe of $\Psi$. Hence, $\text{codim} \, \Ss^{\msf{V}_{\pi/2}}$ here 
is 3 - 0 = 3. 

Now, for a $\Phi$ satisfying its hypotheses, corollary \ref{C:Haus.measure.of.severe.sings.on.spheres} predicts 
$\text{codim} \, \Ss^{\msf{V}_{\pi/2}} \leq \nvar + 1 = 2$. The discrepancy in codimension is explained by the failure of $\Psi$ to satisfy \eqref{E:loc.meas.has.no.severe.sings.near.T}. 
  \end{example}

Consider this weaker version of hypotheses \eqref{E:Phi.on.diagnl.exact} and \eqref{E:loc.measure.has.no.sings.near.T}.
   \begin{equation}  \label{E:location.Phi.on.diagnl.loose} 
     \Phi \text{ is defined and continuous everywhere on $\T$ and for every } 
       y \in S^{\nvar}, \text{ we have } \Phi(y, \ldots, y) \cdot y > -1.
   \end{equation}
(See lemma \ref{L:loose.exactness.of.location.fit} for a generalization.) Here, ``$\cdot$'' indicates the usual inner product on $\mathbb{R}^{\nvar}$. \eqref{E:location.Phi.on.diagnl.loose} will play an important role in chapter \ref{Chptr:aug.direct.mean}. Thus, $\Phi$ is defined everywhere in $\T$, but $\T \cap \Ss$ maybe non-empty. A map $\Phi : \D \partlyto \Ss^{\nvar}$ symmetric in its arguments and satisfying \eqref{E:location.Phi.on.diagnl.loose} will be considered a measure of location on $\Ss^{\nvar}$. 

Note that $\Phi(y, \ldots, y) \cdot y > -1$ if and only if 
$\Phi(y, \ldots, y) + y \neq 0$. I.e., if and only if $\Phi(y, \ldots, y)$ and $y$ are not antipodal.

  \begin{remark}[\eqref{E:location.Phi.on.diagnl.loose} and homotopy]  \label{R:loose.exactness.of.fit.and.homotopy}
We connect \eqref{E:location.Phi.on.diagnl.loose} 
with corollary \ref{C:homotopic.to.identity.on.sphere}. Suppose 
$\Phi : \D \partlyto S^{\nvar}$ satisfies \eqref{E:location.Phi.on.diagnl.loose}.  
Let $\Delta$ be the diagonal map, $\Delta(y) = (y, \ldots, y) \in \T$ ($y \in S^{\nvar}$). Define 
	\begin{equation}  \label{E:arg!.maties!}
	   \arg(z) := |z|^{-1} z \in S^{\nvar} \quad 
		  \bigl( z \in \RR^{\nvar +1} \setminus \{ 0 \} \bigr).
	\end{equation}
$arg$ is continuous on its domain. Then, for $y \in S^{\nvar}$, 
$(\Delta \circ \Delta)(y)= \Delta(y, \ldots, y)$. Tentatively define
   \begin{equation*}
        H(y, t) = 
            \arg \bigl( [1 - t] (\Delta \circ \Delta)(y) + t y \bigr)  \in S^{\nvar},
               \text{ if } y \in S^{\nvar}, \, t \in [0,1] .
   \end{equation*}
We show $H(y, t)$ is defined for all $y \in S^{\nvar}$ and $t \in [0,1]$.
Let $y \in S^{\nvar}$, $t \in [0,1]$. As usual, denote the standard inner product 
on $\RR^{\nvar+1}$ by ``$\cdot$''. Then, by \eqref{E:location.Phi.on.diagnl.loose},
    \begin{multline*}
      \bigl| [1 - t] (\Delta \circ \Delta)(y) + t y \bigr|^{2}
        = (1 - t)^{2} \bigl| (\Delta \circ \Delta)(y) \bigr|^{2} 
          + 2 (1-t) t (\Delta \circ \Delta)(y) \cdot y + t^{2} |y|^{2} \\
            = (1 - t)^{2} + 2 (1-t) t (\Delta \circ \Delta)(y) \cdot y + t^{2} \\
              > (1 - t)^{2} - 2 (1-t) t + t^{2} = (1-2t)^{2} \geq 0 .
    \end{multline*}
Thus, $H$ is defined and continuous on $S^{\nvar} \times [0,1]$. Hence, $H$ is a homotopy between $\Delta \circ \Delta$ and the identity on $S^{\nvar}$ and so \eqref{E:Theta.circ.Delta.homot.to.id.on.sphere} holds.

Therefore, corollary \ref{C:Haus.measure.of.severe.sings.on.spheres} still holds with \eqref{E:Theta.circ.Delta.homot.to.id.on.sphere} replaced by \eqref{E:location.Phi.on.diagnl.loose}. 
  \end{remark} 

  \begin{remark}[Regularization of measures of location on a sphere] \label{R:regularization.on.spheres}
Ideas similar to those in remark \ref{R:regularization.generalities} hold for measuring location on a sphere. But some changes are needed because 
$H_{r}(\D) = \{0\}$, with $r = \nvar$, is not true in that context. Instead we appeal to corollary \ref{C:homotopic.to.identity.on.sphere}. The hypotheses of that corollary must fail for a regularized ``measure of location'', by which we mean a continuous map 
$\Phi : \D \to S^{\nvar}$ which we want to be calibrated (section \ref{SS:calibration}). For such a map, we end up with a hopefully small subset, 
$\T_{\ast}$, on which $\Phi$ must ``unwrap'' $\T$ from around $S^{\nvar}$. 

If a regularized $\Phi$ otherwise behaves reasonably, they by corollary \ref{C:Haus.measure.of.severe.sings.on.spheres} and remark \ref{R:loose.exactness.of.fit.and.homotopy}, it must violate \eqref{E:location.Phi.on.diagnl.loose}. Thus, a regularized measure of location on a sphere has to be maximally uncalibrated at some data set (in $\T$).

A specific example is discussed in remark \ref{R:aug.mean.regularization}.
  \end{remark}   

\chapter{Augmented Directional Mean}  \label{Chptr:aug.direct.mean}
In this chapter we investigate a class of measures of location on spheres (in particular on the circle) called ``augmented directional means''. In chapter \ref{Chptr:robst.loc.on.circle} we apply the results of chapters \ref{Chptr:Haus.meas.of.sing.set} and \ref{Chptr:spherical.location} to show that ``robust'' measures of location on the circle have larger singular sets than do augmented directional means, at least in extreme cases. 

In this chapter (and the next) we assume
    \begin{equation} \label{E:n>2,nvar>0}
        n > 2 \text{ and } \nvar > 0.
    \end{equation}

The ``directional or spherical mean'' (Fisher \emph{et al}  \cite[p.\ 31]{niFtLbjjE87} and 
remark \ref{R:convx.combo.fn.on.sphere} above) is the measure of location, $\Phi_{dm}$, that takes $x := (y_{1}, \ldots, y_{n}) \in \D := (S^{\nvar})^{n}$ to $\bar{x}/|\bar{x}|$, providing $\bar{x} \ne 0$, where $\bar{x}$ is the sample mean of $y_{1}, \ldots, y_{n}$ regarded as vectors in $\mathbb{R}^{\nvar +1}$. I.e., 
	\begin{equation}  \label{E:xbar.defn}
		\bar{x} := n^{-1} (y_{1} + \cdots + y_{n}).
	\end{equation} 

We generalize this somewhat as follows. Let $y_{0} \in S^{\nvar}$ be arbitrary but fixed, let $a \in [0, n)$, and consider the following measure of location. First, let
	\begin{equation}  \label{E:x.bar.y0.a.defn}
		\bar{x}_{y_{0}, a} := \bar{x}_{a} := (a + n)^{-1} ( a y_{0} + n \bar{x} ) .
	\end{equation}
Then let
	\begin{equation} \label{E:aug.direct.mean.defn}
		\mu_{y_{0}, a, n}(x) := \mu_{y_{0}, a}(x) := \mu_{a}(x) 
		    := |\bar{x}_{y_{0}, a}|^{-1} \bar{x}_{y_{0}, a} \in S^{\nvar},
	\end{equation}
whenever $\bar{x}_{y_{0}, a} \ne 0$. $\mu_{a}$ is the directional mean of the data set consisting of $x$ augmented by ``$a$ copies'' of $y_{0}$. (But $a$ does not have to be an integer. In fact, in section \ref{SS:aug.mean.sing.set.size}, we focus 
on the case $n-1 < a < n$.) Call $\mu_{y_{0}, a, n}(x)$ the ``augmented directional mean'' at $x$ (with ``augmentation point'' $y_{0}$ and ``augmentation weight'' $a$).  (This is not to be confused with the method of ``data augmentation'', Tanner \cite{maT91.data.aug}.)

Observe that $\mu_{a}$ is biased toward $y_{0}$. This makes sense if \emph{a priori} one believes the ``true'' location is near $y_{0}$. Larger $a$ corresponds to stronger belief. This idea can be formalized: $\mu_{a}$ is a Bayes estimator (Berger \cite{joB85.BergerBayes}; Nu\~{n}ez-Antonio and Gutierrz-Pena \cite{gN-AeG-P05.BayesDirectData}. This interpretation comes up in remark \ref{R:remoteness.from.Pf}.) Since for $a > 0$, the vector $\bar{x}_{y_{0}, a}$ is just $\bar{x}$ ``shrunk'' toward $y_{0}$, it is tempting to call $\mu_{y_{0}, a}$ a ``directional shrinkage mean'' or something similar. Instead, we call $\mu_{y_{0}, a}$ an ``augmented directional mean''. 

Note that if $a = n$, then, although $(-y_{0}, \cdots, -y_{0}) \in \T$ (see \eqref{E:directional.T.defn}), we have that $\mu_{y_{0}, n}(-y_{0}, \cdots, -y_{0})$ is not defined, so \eqref{E:Phi.on.diagnl.exact} fails. If $a > n$, then \eqref{E:location.Phi.on.diagnl.loose} 
with $\Phi =  \mu_{y_{0}, a}$ and $y = - y_{0}$ fails. On the other hand, let $0 \leq a < n$,  $y_{0} \in S^{\nvar}$, and $x = (y, \ldots, y) \in \T$. 
Then $|a y_{0} + n \bar{x}| = |a y_{0} + n y| \geq n - a > 0$, so $\mu_{y_{0}, a}$ is defined and continuous on $\T$. So, except remark \ref{R:aug.mean.regularization}, we always assume 
	\begin{equation} \label{E:a.tween.0.n}
		0 \leq a < n.
	\end{equation} 

In fact, we \emph{claim:} 
    \begin{equation} \label{E:aug.mean.on.diagnl.loose}
      \text{\eqref{E:location.Phi.on.diagnl.loose} holds for } \Phi = \mu_{a} 
        \text{ with } a \in [0,n) .
    \end{equation} 
Let $a \in [0,n)$. It suffices to show $\mu_{a}(y, \ldots, y) \cdot y > -1$ for every $y \in S^{\nvar}$. For suppose not. Then there exists $y \in S^{\nvar}$ s.t.\ $\mu_{y_{0}, a}(y, \ldots, y) \cdot y \leq -1$. This cannot happen if $a=0$, so $a > 0$. By the (Cauchy-)Schwarz inequality (Stoll and Wong \cite[Theorem 3.1, p.\ 79]{rrSetW68.LinearAlgebra}), 
	\begin{equation*}
		-1 = -\bigl| \mu_{y_{0}, a}(y, \ldots, y) \bigr| |y| 
		  \leq \mu_{y_{0}, a}(y, \ldots, y) \cdot y \leq -1.
	\end{equation*}
So $\mu_{y_{0}, a}(y, \ldots, y) \cdot y = -1$. By (Cauchy-)Schwarz again we have that 
$\mu_{y_{0}, a}(y, \ldots, y)$ and $y$ are linearly dependent. 
Hence $\mu_{y_{0}, a}(y, \ldots, y) = -y$. Let $c := (a+n) |\bar{x}_{y_{0}, a}| > 0$. Then 
    \begin{equation*}
        a y_{0} + n y =(a+n)  \bar{x}_{y_{0}, a} = c \mu_{y_{0}, a}(y, \ldots, y) = -c y. 
    \end{equation*}
Therefore, $y_{0} = \pm y$. If $y_{0} = y$ then $a = -n-c < 0$. Therefore, $y_{0} = -y$. But this means $n > a = n+c > n$, another contradiction. Thus,  \eqref{E:aug.mean.on.diagnl.loose} holds for $\Phi = \mu_{a}$.

For simplicity we often assume, WLOG, that 
	\begin{equation} \label{E:y0.starts.with.0s}
		y_{0} := (0, \ldots, 0, 1) \in S^{\nvar} \subset \RR^{\nvar +1}.
	\end{equation}
Observed that,
	\begin{multline} \label{E:same.hemisphere}
	 \text{If \eqref{E:y0.starts.with.0s} holds and } y \in S^{\nvar}
	   \text{ lies in the same hemisphere as } -y_{0} \\
	 \text{ then for some } w \in \overline{B_{1}^{\nvar}(0)} 
  \text{ we have } y = \bigl( w, -\sqrt{1 - | w |^{2} } \bigr)
	\end{multline}

Assume \eqref{E:y0.starts.with.0s}. Write $x_{0} := (y_{0}, \ldots, y_{0}) \in \D$. Let $x = (y_{1}, \ldots, y_{n}) \in \D$. Write $y_{i} = (w_{i}, z_{i})$, where $w_{i} \in \overline{B_{1}^{\nvar}(0)}$ (see \eqref{E:Euc.ball.defn}) and $z_{i} \in \RR$ ($i=1, \ldots,, n$). 
Thus, $|w_{i}|^{2} + |z_{i}|^{2} = 1$. 
Suppose $\bigl| x - (- x_{0}) \bigr| \leq \sqrt{2}$. Let $i = 1, \ldots, n$. Then
    \begin{equation*}
      2 - 2 y_{i} \cdot (-y_{0}) = |y_{i} - (-y_{0})|^{2} \leq 2, 
        \; i = 1, \ldots, n .
    \end{equation*}
Thus, $z_{i} = - y_{i} \cdot (-y_{0}) \leq 0$ so $y_{i}$ lies in the same hemisphere 
as $-y_{0}$, for each $i=1, \ldots, n$. Therefore, 
	\begin{multline}  \label{E:y.in.terms.of.w}
	    \text{Let } x = (y_{1}, \ldots, y_{n}) \in \D. 
	      \text{ If } |x - (-x_{0})| \leq \sqrt{2} \, \\
	        \text{ then }
	          \; y_{i} = \bigl( w_{i}, -\sqrt{1 - | w_{i} |^{2} } \bigr), 
	            \quad i =  1, \ldots, n.
	\end{multline}

\section{The singularities of $\mu_{a}$} \label{SS:sings.of.directional.mean}
Let 
    \begin{multline}   \label{E:S'a.D'a.defns}
      \Ss'_{a} := \{ x \in \D : \bar{x}_{a} = 0 \} 
        = \left\{ (y_{1}, \ldots, y_{n}) \in (S^{\nvar})^{n} : 
          \sum_{1}^{n} y_{i} = -a y_{0} \right\}
          \text{ and } \\ 
        \D'_{a} := \D' := \D \setminus \Ss'_{a} 
          = \left\{ (y_{1}, \ldots, y_{n}) \in (S^{\nvar})^{n} : 
            \sum_{1}^{n} y_{i} \neq -a y_{0} \right\} .
    \end{multline}
First, note that if $x = (y_{1}, \ldots, y_{n}) \in \D$ with $\bar{x}_{a} \ne 0$ then 
$\mu_{a}(x)$ is defined and continuous in a neighborhood of $x$. (In particular, 
if $y \in S^{\nvar}$ then $\mu_{a}$ is defined and continuous in a neighborhood 
of $x := (y, \ldots, y)$.) Thus, $\mu_{a}$ is continuous on $\D'_{a}$.
Obviously, $\Ss'_{a}$ is closed and \eqref{E:assume.Phi.S'.sym} holds 
for $\Phi = \mu_{a}$. Note that, by \eqref{E:x.bar.y0.a.defn}, if $y \in S^{\nvar}$ and $x = (y, \ldots, y) \in \T$ then, (Cauchy-)Schwarz, 
    \begin{multline}  \label{E:|x.bar.y0,a|.sqrd}
      |\bar{x}_{y_{0}, a}|^{2} = |a + n|^{-2} |a y_{0} + n y|^{2} 
        \geq |a + n|^{-2} \bigl( a^{2} - 2 an (y_{0} \cdot y) + n^{2} \bigr) \\
          \geq |a + n|^{-2} (a^{2} - 2 an + n^{2}) 
            = |a + n|^{-2} (a-n)^{2}.
    \end{multline}
Thus, by \eqref{E:a.tween.0.n},
    \begin{equation} \label{E:Ss'.a.cap.T.is.empty}
      \Ss'_{a} \cap \T = \varnothing , \text{ if } a < n .
    \end{equation}

  \begin{remark}[Regularization of $\mu_{a}$] \label{R:aug.mean.regularization}
Consider the map $\Psi : \D \times [0, \infty) \partlyto S^{\nvar}$ defined by 
$\Psi(x, a) := \mu_{a}(x)$, whenever it is defined. $\Psi$ is an instance of the general map $\Psi$ discussed in remark \ref{R:regularization.generalities}. \eqref{E:|x.bar.y0,a|.sqrd} implies $\Ss'_{a} \cap \T = \varnothing$ if $a > n$. In fact, if $a > n$, then
    \begin{equation*}
      \left| a y_{0} + \sum_{i=1}^{n} y_{i} \right| 
        \geq a - \left| \sum_{i=1}^{n} y_{i} \right| \geq a - n > 0 .
    \end{equation*}
Therefore, $\mu_{a}$ is defined and continuous on all of $\D$ if $a > n$. Thus, the augmented directional mean allows regularization in a straightforward way. (See remark \ref{R:regularization.on.spheres}.)

Let $a \geq 0$. Let $y \in S^{\nvar}$. We may write 
$y = y(\theta) := (\cos \theta) y_{0} + (\sin \theta) z$, where $z \in S^{\nvar}$ is orthogonal to $y_{0}$ and $\theta \in [0 , \pi]$. 
Let $x = x(\theta) = (y(\theta), \ldots, y(\theta)) \in \T$. 
For some $\phi = \phi(\theta) \in (-\pi , \pi]$ we have 
$\mu_{a}(x) = (\cos \phi) y_{0} + (\sin \phi) z$. Thus,
    \begin{equation}  \label{E:mu.a(x).dot.y}
      \mu_{a} \bigl( x(\theta) \bigr) \cdot y(\theta) = \cos(\phi - \theta) .
    \end{equation} 
$\mu_{a} \bigl( x(\theta) \bigr)$ is proportional 
to $(a + n \cos \theta) y_{0} + (n \sin \theta) z$. 
For $\omega \in (-\pi, \pi]$ and $r > 0$ define \linebreak
$\arg(r \cos \omega, r \sin \omega) := \omega$. Thus, $\arg$ is continuous and
    \begin{equation*}  
      \phi(\theta) = \arg (a + n \cos \theta , \, n \sin \theta) 
        = \arg (a/n + \cos \theta , \, \sin \theta) .
    \end{equation*}
By comparing $\phi$ to $\theta$ we can see how poorly calibrated $\mu_{a}$ is. The preceding implies that if $a = r n$, with $r \geq 0$ constant, then $\phi(\theta)$ is independent of $n$.

This is illustrated in figure \ref{F:aug.mean.calibration}. The red lines are the graphs of $\phi(\theta)$ for a measure of location $\Phi$ on the sphere satisfying \eqref{E:Phi.on.diagnl.exact}, e.g.\ the directional mean $\mu_{0}$. One has to sacrifice that when $a > 0$. In the top plot, with $a = 0.99 \, n$, we have that 
$\bigl| \phi(\theta) - \theta \bigr| < \pi$. Therefore, by \eqref{E:mu.a(x).dot.y}, 
for $a = 0.99 \, n$, we have $\mu_{a}(x) \cdot y > -1$. I.e., $\mu_{0.99 \, n}$ satisfies, \eqref{E:location.Phi.on.diagnl.loose} and by corollary \ref{C:Haus.measure.of.severe.sings.on.spheres} and remark \ref{R:loose.exactness.of.fit.and.homotopy}, $\mu_{0.99 \, n}$ has singularities.  

But in the lower plot, showing the graph 
of $\phi(\theta)$ when $a = 1.01 \, n$, we have $\phi(\pi) - \pi = -\pi$ and \eqref{E:location.Phi.on.diagnl.loose} fails. Since $1.01 \, n > n$, as observed above, 
$\mu_{1.01 \, n}$ is continuous, i.e., has no singularities. In the bottom panel, as 
$\theta \uparrow \pi$ the black curve, instead of proceeding up to $\pi$, returns to 0, leading to the grossest of possible errors: $\mu_{1.01} \bigl( x(\pi) \bigr) = - y(\pi)$. This is the kind of ``unwrapping'' described in remark \ref{R:regularization.generalities}. The difference in the graphs is dramatic and the switch, which occurs at $a = n$, from one pattern to the other is the sort of bifurcation also mentioned in remark \ref{R:regularization.generalities}.
  \end{remark}
  
   \begin{figure}
      \epsfig{file = 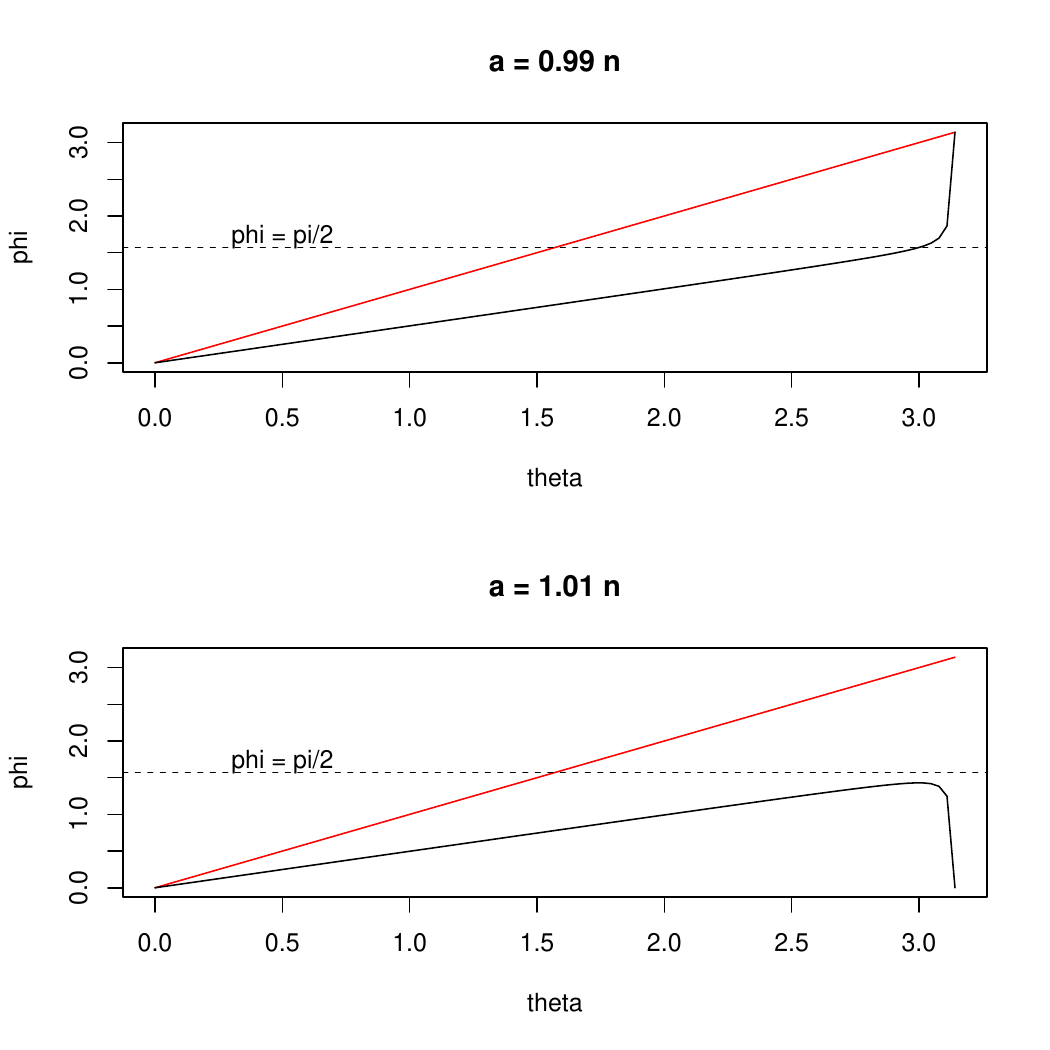, height = 5.8in, , }  
      \caption{$\theta$-axis is the angular displacement from $y_{0}$ of a point $y$ on $S^{\nvar}$. $\phi$-axis is the angular displacement from $y_{0}$ of the image under a measure of location on the sphere applied to $(y, \ldots, y) \in \T$. Red lines are the identity. That is the $\phi$ vs. $\theta$ curve of a perfectly calibrated measure of location, i.e., one satisfying \eqref{E:Phi.on.diagnl.exact}. The black curves are the $\phi$ vs.\ $\theta$ curve of the  augmented mean with, in the top graph, augmentation weight $0.99 n$ and, in the bottom, weight $1.01 n$. The augmented mean is regularized in the bottom plot, but not in the top one.
The two graphs illustrate the bifurcation that occurs in $\mu_{a}$ at $a = n$. }
             \label{F:aug.mean.calibration}
  \end{figure}

Denote the singular set of $\mu_{a}$ w.r.t.\ $\D'_{a}$ by $\Ss_{\mu_{y_{0}, a, n}}$, 
$\Ss_{y_{0}, a, n}$, $\Ss_{y_{0}, a}$, or $\Ss_{a}$. Then $\Ss'_{a}$ is a closed superset of $\Ss_{a}$. Denote the set of $\msf{V}_{\pi/2}$-severe singularities of $\mu_{a}$ 
by $\Ss_{a}^{\msf{V}_{\pi/2}}$. (See section \ref{S:convex.combos.on.spheres}.) To keep things simple assume \eqref{E:a.tween.0.n}. \emph{Claim:} 
	\begin{equation}  \label{E:Sa=Sa'}
	  \Ss'_{a} = \Ss_{a} = \Ss_{a}^{\msf{V}_{\pi/2}} .
	  \end{equation}
So $\Ss_{a}$ is an algebraic variety. 
 
Let $x = (y_{1}, \ldots, y_{n}) \in \Ss'_{a}$ be arbitrary. Let $v \in S^{\nvar}$ be orthogonal to $y_{1}$ and for $\theta \in \RR$, let 
$y_{1}' := y_{1}'(\theta) := \cos \theta \, y_{1} + \sin \theta \, v$. 
Thus, $y_{1}' \in S^{\nvar}$. 
For $i = 2, \ldots, n$ let $y_{i}' := y_{i}'(\theta) :=y_{i}$, so if $i > 1$ then $y_{i}'(\theta)$ is constant in $\theta$, and let $x'(\theta) := x'(\theta,v) := (y_{1}', y_{2}', \ldots, y_{n}')$. 
Let $\bar{x}'_{a} := (a + n)^{-1} \bigl( a y_{0} + \sum_{i=1}^{n} y_{i}' \bigr)$. 
Since $x \in \Ss'_{a}$,
    \begin{equation}  \label{E:x'.bar.for.trig.x'}
      (a + n) \bar{x}'_{a} = a y_{0} + \sum_{i=1}^{n} y_{i}'(\theta) 
        = a y_{0} + (y_{1}' - y_{1}) + \sum_{i=1}^{n} y_{i} 
          = y_{1}' - y_{1} = -(1 - \cos \theta) y_{1} + \sin \theta \, v .
    \end{equation}
A bi-product of the preceding is the following. Suppose $\theta \in (0, \pi/2)$, so $\theta \neq 0$. Then $\bar{x}'_{a} \neq 0$. Thus, $x'(\theta) \notin \Ss'_{a}$. But for such $\theta$, $x'(\theta)$ can be arbitrarily close to $x$. We conclude,
    \begin{equation} \label{E:mu.a.sing.set.has.empty.interior}
      \Ss'_{a} \text{ has empty interior.} 
    \end{equation}
(Which is essentially immediate from \eqref{E:S'a.D'a.defns} anyway.) Therefore, by \eqref{E:S'a.D'a.defns}, 
    \begin{equation} \label{E:mu.a.satisfies.E:D'.dense}
      \D' \text{ is dense in } \D .
    \end{equation}

By \eqref{E:x'.bar.for.trig.x'}, 
    \begin{equation*}
      \left| a y_{0} + \sum_{i=1}^{n} y_{i}'(\theta) \right|^{2}
      = \cos^{2} \theta + \sin^{2} - 2 \cos \theta + 1 = 2(1 - \cos \theta) ,
    \end{equation*}
because $y_{1} \cdot v = 0$ by choice of $v \in S^{\nvar}$. Therefore, by \eqref{E:x'.bar.for.trig.x'} again, 
    \begin{multline}  \label{E:mu.a.for.x'(theta)}
      \mu_{a} \bigl( x'(\theta) \bigr) = |\bar{x}'_{a}|^{-1} \bar{x}'_{a}
        = \frac{1}{\sqrt{2(1 - \cos \theta)}}
           \bigl] -(1 - \cos \theta) y_{1} + \sin \theta \, v \bigr] \\
            = -(1 - \cos \theta) y_{1} + \sin \theta \, v
              = - \frac{\sqrt{1 - \cos \theta}}{\sqrt{2}} \, y_{1} 
               + \frac{\sin \theta}{\sqrt{2(1 - \cos \theta)}} \, v .
    \end{multline}
But, applying L'Hospital's rule (Rudin \cite[Theorem 5.13, p.\ 94]{wR64.PMA}), we get 
    \begin{equation*}
       \frac{\sin \theta}{\sqrt{2(1 - \cos \theta)}} = \frac{\sin \theta}{2\sqrt{(1 - \cos \theta)/2}}
         = \frac{\sin \theta}{2 \sin \tfrac{1}{2} \theta} \to 1 \text{ as } \theta \to 0 .
    \end{equation*}
Substituting this into \eqref{E:mu.a.for.x'(theta)} and letting $\theta \to 0$, we get
$\mu_{a}(x'(\theta)) \to v$. I.e., 
    \begin{multline} \label{E:every.pt.in.S'a.is.really.bad.sing}
       \text{In an arbitrary neighborhood of any } x = (y_{1}, \ldots, y_{n}) \in \Ss'_{a} \\
         \text{ there is a data set $x' \in \D'$ s.t.\ $\mu_{a}(x')$ is arbitrarily close } \\
           \text{to any unit vector orthogonal to } y_{1}. 
    \end{multline}
In particular, arbitrarily close to any $x \in \Ss'_{a}$ there are data sets in $\D'$ whose images under $\mu_{a}$ are arbitrarily close to being antipodal. Thus, the closure of the image of any neighborhood of $x \in \Ss'_{a}$ in $\D'$ lies in no $V \in \msf{V}_{\pi/2}$.
Thus, $x \in \Ss_{a}^{\msf{V}_{\pi/2}} \subset \Ss_{a} \subset \Ss'_{a}$ (section \ref{S:convex.combos.on.spheres}). But $x$ is an arbitrary data set in $\Ss'_{a}$. 
The claim \eqref{E:Sa=Sa'} is proved. Examples of singularities of $\mu_{a}$ for two different values of $a$ are shown in figure \ref{F:CircleSing}. Those examples are examined in detail in section \ref{SS:17.point.aug.mean.exml.sings}. 

It now follows from \eqref{E:SV.is.closed} that
        \begin{equation}   \label{E:mu.a.Sa.is.compact}
            \Ss_{a} \text{ is compact.}
        \end{equation} 
It now follows from \eqref{E:Ss'.a.cap.T.is.empty} that  
    \begin{equation} \label{E:S.a.cap.T.empty}
      \Ss^{\msf{V}_{\pi/2}} \cap \T = \Ss_{a} \cap \T = \varnothing .
    \end{equation}
It follows from \eqref{E:mu.a.sing.set.has.empty.interior}, \eqref{E:mu.a.Sa.is.compact}, and \eqref{E:S.a.cap.T.empty} that \eqref{E:loc.measure.has.no.sings.near.T} holds 
with $\Ss' = \Ss_{a}$. \emph{A fortiori}, \eqref{E:loc.meas.has.no.severe.sings.near.T} holds for $\mu_{a}$. 

Recalling \eqref{E:aug.mean.on.diagnl.loose}, in summary we have,
	\begin{equation}  \label{E:basic.props.of.mu.a}
           \text{\eqref{E:location.Phi.on.diagnl.loose}, \eqref{E:loc.measure.has.no.sings.near.T}, 
		\eqref{E:loc.meas.has.no.severe.sings.near.T}
		and \eqref{E:assume.Phi.S'.sym} all hold for } 
		   \Phi = \mu_{a} \text{ for } a \in [0, n).
	\end{equation}
Therefore, by corollary \ref{C:Haus.measure.of.severe.sings.on.spheres} and remark \ref{R:loose.exactness.of.fit.and.homotopy}, 
we have \eqref{E:Haus.measure.severe.sings.on.spheres} holds for $\Phi = \mu_{a}$ and 
$\Ss = \Ss_{a}$.   
Let $R_{a} > 0$ denote the $\Hm^{n \nvar-\nvar-1}$-essential distance, 
$dist_{n \nvar-\nvar-1}(\Ss_{a}, \T) \geq dist (\Ss_{a}, \T) > 0$, 
from $\Ss_{a}$ to $\T$. (See \eqref{E:essential.dist.defn}.)
Then, by \eqref{E:Sa=Sa'},
    \begin{multline}  \label{E:R.n.nvar-p-1.H.bdd.away.frm.0}
      R_{a}^{-(n \nvar-\nvar-1)} \Hm^{n \nvar-\nvar-1}(\Ss_{a}) 
        = R_{a}^{-(n \nvar-\nvar-1)} \Hm^{n \nvar-\nvar-1}(\Ss_{a}^{\msf{V}_{\pi/2}}) \\
		      \text{ is bounded away from 0 as } a \uparrow n. \\
            \text{ In particular, }
         \dim \Ss_{a} = \dim \Ss_{a}^{\msf{V}_{\pi/2}} 
          \geq n \nvar - \nvar - 1.
    \end{multline}

\section{Size of singular set of augmented directional mean} \label{SS:aug.mean.sing.set.size}
(For analysis of the singular set of the augmented directional mean in a specific case, see section \ref{SS:17.point.aug.mean.exml.sings}.) Let 
    \begin{equation*}
      \Ss_{y_{0}, a, n} := \Ss_{a, n} := \Ss_{a}
    \end{equation*}
be the singular set of $\mu_{y_{0}, a}$. We have just seen that 
$\text{codim} \, \Ss_{a, n} \leq \nvar + 1$. First we prove the \emph{Claim:}   
    	\begin{equation}   \label{E:codim.Sa.=.nvar+1}
    		\text{codim} \, \Ss_{a, n} = \nvar + 1 .
    	\end{equation}
I.e., the augmented directional mean achieves the bound \eqref{E:bound.on.sing.set.dim.in.sphere.loc}.   

 Let
	   \begin{multline} \label{E:not.single.line.U.defn}
	      U := \bigl\{ (y_{1}, \ldots, y_{n}) \in (\RR^{\nvar +1})^{n} : 
	        \text{No } y_{i} = 0 \; (i = 1, \ldots, n)   \\
	          \text{ and } y_{1}, \ldots, y_{n} 
	            \text{ do not all lie on a single line through the origin in } 
	              \RR^{\nvar +1}  \bigr\}.
	   \end{multline}
Thus, $U$ is an open subset of $(\RR^{\nvar +1})^{n}$. 
Let 
	\begin{equation*}
		\tilde{\Ss}_{a} := \tilde{\Ss}_{a, n} := \Ss_{a} \cap U.
	\end{equation*}

Let $x := (y_{1}, \ldots, y_{n}) \in \Ss_{a} \setminus U$. Since $0 \notin \Ss_{a}$, we must have that $y_{1}, \ldots, y_{n} \in \D$ all lie on a single line through the origin 
in $\RR^{\nvar +1}$. On a given line through the origin there are $2^n$ data sets, singular or not, because there are two unit vectors (they are antipodal) that each span the line. Thus, $y_{i} = \pm y_{1}$ ($i = 2, \ldots, n$). By \eqref{E:Sa=Sa'} and \eqref{E:S'a.D'a.defns}, we also have $\sum_{i} y_{i} = -a y_{0}$. Therefore, 
$y_{i} = \pm y_{0}$ ($i = 2, \ldots, n$) and $a$ is an integer. It follows that 
$\Ss_{a} \setminus U$ is at most finite. In fact, if $a$ is not an integer 
$\Ss_{a} \setminus U$ is empty. Thus, by \eqref{E:n>2,nvar>0},
    \begin{equation} \label{E:Hm.S.less.U.0}
        \text{If $a$ is not an integer } \Ss_{a} \setminus U = \varnothing 
          \text{ so } \Ss_{a} = \tilde{\Ss}_{a}. \text{ In any case, } 
            \Hm^{n \nvar-\nvar-1}(\Ss_{a} \setminus U) = 0
    \end{equation}
Hence, 
    \begin{equation}  \label{E:Hm.S.tilde.=.Hm.S}
      \Hm^{n \nvar-\nvar-1}(\tilde{\Ss}_{a}) = \Hm^{n \nvar-\nvar-1}( \Ss_{a}) .
    \end{equation}
Thus, it sufffices to show $\text{codim} \, \tilde{\Ss}_{a} = \nvar +1$. 
Consider the map $H : U \to \RR^{n + \nvar + 1}$ given by
   \[
      H(y_{1}, \ldots, y_{n}) 
         := \bigl( y_{0} + \sum_{i=1}^{n} y_{i}, 
           \, |y_{1}|^{2}, \ldots, |y_{n}|^{2} \bigr)^{1 \times (n+\nvar+1)}, 
             \quad (y_{1}, \ldots, y_{n}) \in U.
   \]
 We have,
    \begin{equation} \label{E:S.tilde.a.is.invrs.image}
      \tilde{\Ss}_{a} = H^{-1} \bigl( (1-a)y_{0}, 1, \ldots, 1 \bigr)^{1 \times (n+1)} . 
    \end{equation} 
By \eqref{E:R.n.nvar-p-1.H.bdd.away.frm.0} and \eqref{E:Hm.S.tilde.=.Hm.S}, 
$\dim \tilde{\Ss}_{a,n} \geq n \nvar - \nvar -1$. In particular,
 $\tilde{\Ss}_{a,n} \ne \varnothing$ 
 so $\bigl( (1-a)y_{0}, 1, \ldots, 1 \bigr)^{1 \times (n+1)} \in H(U)$. 
Regarding each $y_{i}$ as a $1 \times (\nvar +1)$ row matrix, the Jacobian matrix of $H$ (Boothby \cite[p.\ 26]{wmB75}) is given by
   \[
      DH(y_{1}, \ldots, y_{n})^{(n+\nvar +1) \times n(\nvar +1)} =
         \begin{pmatrix}
            I_{\nvar +1} & I_{\nvar +1} & \cdots & I_{\nvar +1} \\
            2 y_{1} & 0^{1 \times (\nvar +1)} & \cdots & 0^{1 \times (\nvar +1)} \\
            0^{1 \times (\nvar +1)} & 2 y_{2} & \cdots & 0^{1 \times (\nvar +1)} \\
            \vdots &  \vdots & \ddots & \vdots \\
            0^{1 \times (\nvar +1)} & 0^{1 \times (\nvar +1)} & \cdots & 2 y_{n}
         \end{pmatrix}^{(n+\nvar+1) \times n(\nvar+1)} .
   \]
Subtracting the first $\nvar +1$ columns of this matrix as a block from the remaining columns we see that $DH$ has the same rank as
   \begin{equation*}   \label{E:I.on.top.2x1.diag}
      \begin{pmatrix}
         I_{\nvar +1} & 0^{(\nvar +1) \times (\nvar +1)} & \cdots & 
           0^{(\nvar +1) \times (\nvar +1)} \\
         2 y_{1} & -2 y_{1} & \cdots & -2 y_{1} \\
         0^{1 \times (\nvar +1)} & 2 y_{2} & \cdots & 0^{1 \times (\nvar +1)} \\
         \vdots &  \vdots & \ddots & \vdots \\
         0^{1 \times (\nvar +1)} & 0^{1 \times (\nvar +1)} & \cdots & 2 y_{n}
      \end{pmatrix}.
   \end{equation*}

Subtracting the appropriate linear combination of the first $\nvar +1$ rows from the 
$(\nvar +2)^{nd}$ row of \eqref{E:I.on.top.2x1.diag}
we see that $DH$ has the same rank as
    \begin{equation}  \label{E:I.0.on.top.2x1.diag}
      \begin{pmatrix}
         I_{\nvar +1} & 0^{(\nvar +1) \times (\nvar +1)} & \cdots & 0^{(\nvar +1) 
           \times (\nvar +1)} \\
         0^{1 \times (\nvar +1)} & -2 y_{1} & \cdots & -2 y_{1} \\
         0^{1 \times (\nvar +1)} & 2 y_{2} & \cdots & 0^{1 \times (\nvar +1)} \\
         \vdots &  \vdots & \ddots & \vdots \\
         0^{1 \times (\nvar +1)} & 0^{1 \times (\nvar +1)} & \cdots & 2 y_{n}
      \end{pmatrix}.
   \end{equation}
Consider the last $n$ rows of the preceding. Now, $(y_{1}, \ldots, y_{n}) \in U$ so none of the rows is 0. A nontrivial linear combination of those $n$ rows has the form 
$z = (0^{1 \times (\nvar +1)}, -a_{1} y_{1} + a_{2} y_{2}, \ldots, 
-a_{1} y_{1} + a_{n} y_{n})$ with $a_{1}, \ldots, a_{n} \in \RR$ not all 0. Suppose $z = 0$. If $a_{1} = 0$ then for at least one $j = 2, \ldots, n$ we have $a_{j} \neq 0$ so $-a_{1} y_{1} + a_{j} y_{j} = a_{j} y_{j} \neq 0$, contradicting $z=0$. So assume $a_{1} \neq 0$. 
$z = 0$ then implies none of $a_{2}, \ldots, a_{n}$ are 0 and, for $j = 2, \ldots, n$, we have $y_{j} = (a_{1}/a_{j}) y_{1}$. I.e., $y_{1}, \ldots, y_{n}$ all lie on the same line, viz., the line spanned by $y_{1}$. This contradicts $(y_{1}, \ldots, y_{n}) \in U$. Therefore the last $n$ rows are linearly independent and the matrix \eqref{E:I.0.on.top.2x1.diag} has full rank 
$n + \nvar + 1$. 

That means 
$DH(y_{1}, \ldots, y_{n})^{(n+\nvar +1) \times n(\nvar +1)}$ also has full rank 
$n + \nvar + 1$. Now, $U$ is open in $(\RR^{\nvar +1})^{n}$. Therefore, by \eqref{E:S.tilde.a.is.invrs.image} and 
Boothby \cite[Theorem (5.8), p.\ 79]{wmB75} and \eqref{E:Hm.S.less.U.0}, 
	\begin{multline}  \label{E:dim.Sa.tilde}
	    \tilde{\Ss}_{a} \text{ is a closed, regular submanifold of } U \text{ and } \\
	      \dim (\tilde{\Ss}_{a}) = \dim U - (n + \nvar + 1) = n(\nvar +1) -n - \nvar -1 \\
	      = n \nvar - \nvar - 1 = \dim \D - (\nvar +1).
	\end{multline}
We had already observed that it sufficed to show 
$\text{codim} \, \tilde{\Ss}_{a} = \nvar +1$. Therefore we have proved the claim \eqref{E:codim.Sa.=.nvar+1} that the augmented mean achieves bound \eqref{E:bound.on.sing.set.dim.in.sphere.loc}.

Next, we examine what happens to $\Hm^{n \nvar - \nvar-1}(\Ss_{\mu_{y_{0}, a, n}})$ as $a \uparrow n$. Continue to assume \eqref{E:n>2,nvar>0}. Let 
	\begin{equation}  \label{E:a.tween.n-1.n}
		a \in (n-1, n)
	\end{equation}
and suppose \eqref{E:y0.starts.with.0s} holds. Let $x = (y_{1}, \ldots y_{n}) \in \D$. Then, by \eqref{E:Sa=Sa'},
	\begin{equation}    \label{E:sum.tilde.y.=.-a.y0}
		x \in \Ss_{a} \text{ if and only if } 
		          \sum_{i=1}^{n} y_{i} = -a y_{0} = (0, \ldots, 0, - a).
	\end{equation}
Now, $\T$ is compact and $\Ss_{a}$ is compact by \eqref{E:mu.a.Sa.is.compact}. 
Therefore, by \eqref{E:S.a.cap.T.empty}, we have 
    \begin{equation} \label{E:dist.Sa.T.>.0}
      dist (\Ss_{a}, \T) > 0.
    \end{equation}
By \eqref{E:a.tween.0.n} and \eqref{E:Sa=Sa'}, $\mu_{a}$ satisfies \eqref{E:loc.meas.has.no.severe.sings.near.T} (with $\mu_{a}$ in place of $\Phi$).

Let $x \in \Ss_{a}$. Let $i = 1, \ldots, n$. For $y_{i} \in \RR^{\nvar +1}$, let the $j^{th}$ coordinate of $y_{i}$ be denoted by $y_{i, j}$ or $y_{ij}$ ($j=1, \ldots, \nvar+1$). Now, for every $j$ we have $y_{j,\nvar +1} \geq -1$. Thus, replacing $y_{j,\nvar +1}$ ($j \neq i$) by $-1$ we get, by \eqref{E:sum.tilde.y.=.-a.y0}, 
	\begin{equation}   \label{E:with.n-1.1s}
		-a \geq y_{i,\nvar +1} - (n-1).
	\end{equation}
Therefore, if we let  
	\begin{equation}  \label{E:delta.a.defn}
		\delta := \delta_{a} := a - (n-1),
	\end{equation}
then, by \eqref{E:a.tween.n-1.n} and \eqref{E:with.n-1.1s}, we have,
	\begin{equation}  \label{E:tilde.y.i.nvar +1.bdd.away.from.0}
		\delta \in (0,1) \text{ and } y_{i,\nvar +1} \leq -\delta < 0, 
		  \quad i = 1, \ldots, n.
	\end{equation}
Let $w_{i} = (y_{i1}, \ldots, y_{i\nvar})$. Then by \eqref{E:y0.starts.with.0s}, we have that \eqref{E:same.hemisphere} holds with $y = y_{i}$ and $w = w_{i}$ and 
	\begin{equation}  \label{E:1-|w|2.>.delta}
		\sqrt{1 - | w_{i} |^{2} } \geq \delta_{a} > 0, \quad i =  1, \ldots, n.
	\end{equation}
Hence, $1 > 1 - \delta_{a}^{2} \geq | w_{i} |^{2}$. Since $a \in (n-1, n)$, it follows from \eqref{E:delta.a.defn} and \eqref{E:a.tween.n-1.n} that
	\begin{align}  \label{E:w.bdd.away.from.ball.bndry}
	  |w_{i}|^{2} &\leq 1 - \bigl[ (n-1) -a \bigr]^{2} \notag \\
	    &= 1 - \bigl[ (n-a) -1 \bigr]^{2} \\
	    &= (n-a) \bigl[ 2 - (n-a) \bigr] < 2(n-a) , \notag \\
	    & \qquad i =  1, \ldots, n . \notag 
	\end{align}

Conversely, by \eqref{E:sum.tilde.y.=.-a.y0}, since $x \in \Ss_{a}$ and $a < n$, not all $y_{i}$'s equal $-y_{0}$. In fact, 
	\begin{multline}  \label{E:minmax.tilde.w}
		\text{For \emph{some} } j = 1, \ldots, n, \text{ we have } 
		\sqrt{1 - | w_{j} |^{2}} = -y_{j,\nvar+1} \leq a/n, \\
		  \text{ so } | w_{j} | \geq \sqrt{1 - a^{2}/n^{2}}.
	\end{multline}

We are interested in the singular set $\Ss_{a}$ of $\mu_{a}$ as $a \uparrow n$. Now, by \eqref{E:y0.starts.with.0s} and \eqref{E:y.in.terms.of.w}, 
    \begin{equation*}
         \bigl| y_{i} - (-y_{0}) \bigr|^{2} 
           = |w_{i}|^{2} + \bigl( 1 - \sqrt{1 - | w_{i} |^{2} } \bigr)^{2}
             = 2 - 2 \sqrt{1 - | w_{i} |^{2} } .
    \end{equation*}
Since $x = (y_{1}, \ldots, y_{n})$ is an arbitrary element of $\Ss_{a}$ and for each $i$, $y_{i} = (w_{i}, - \sqrt{1 - | w_{i} |^{2}} )$, we have, by the preceding,
    \begin{multline} \label{E:dist(-y0,Sa).min.max.range}
      \min_{x \in \Ss_{a}} \left( 2 - 2 \sqrt{1 - | w_{i} |^{2} } \right) 
        \leq \bigl| -(y_{0}, \ldots, y_{0}) - x \bigr|\
          \leq \max_{x \in \Ss_{a}} \left( 2 - 2 \sqrt{1 - | w_{i} |^{2} } \right)  \\
            \text{ for every } x \in \Ss_{a} .
    \end{multline}

For every $i = 1, \ldots, n$ we have $\sqrt{1 - | w_{i} |^{2}} \leq 1$. Let $j$ be as in \eqref{E:minmax.tilde.w}. Then, by \eqref{E:dist(-y0,Sa).min.max.range}, 
\eqref{E:delta.a.defn}, and \eqref{E:1-|w|2.>.delta}, 
    \begin{multline}  \label{E:dist.y0s.to.Sa}
         \sqrt{ 2 \frac{n-a}{n} } 
           = \sqrt{ \sum_{i \neq j} (2 - 2) + (2 - 2 a/n) } \\
           \leq \bigl| -(y_{0}, \ldots, y_{0}) - x \bigr|
             \leq \sqrt{ 2n(1 - \delta_{a}) } = \sqrt{2n(n-a)} 
             \text{ for every } x \in \Ss{a} .
    \end{multline}
Therefore, by \eqref{E:dist.Sa.T.>.0}, as $a \uparrow n$, the distance in 
$\RR^{n(\nvar + 1)}$ from $\Ss_{a}$ to $\T$ remains strictly positive, but goes to 0. 

Next, we bound above the LHS of \eqref{E:R.n.nvar-p-1.H.bdd.away.frm.0} above. Here we hold $a \in (n-1,n)$ fixed and use $\mu_{a}$ as a kind of template while we consider alternative values of $a$, which we denote by $t$. We will make use of Landau ``$O$'', ``$o$'' notation 
(de Bruijn \cite[Sections 1.2 and 1.3]{ngdeB81.AsympMthdsAnlys}). Part \ref{I:aug.mean.R.vol.bounded} of the following is used in the proof of proposition \ref{P:aug.direct.mean.beats.robst.loc}. See appendix \ref{Chptr:Lip.Haus.meas.dim} for the definition of ``locally Lipschitz''. See \eqref{E:set.distances} and \eqref{E:essential.dist.defn} for the definition of distance and essential distance. For proof of the following see appendix \ref{Chptr:misc.proofs}.
  \begin{prop}  \label{P:rate.of.decrease.of.Hm.St.aug.direct.mean}
Assume $n > 2$. Let $t \in [0, n)$ and $x = (y_{1}, \ldots, y_{n}) \in \Ss_{t}$. We have the following:
  \begin{enumerate}
	\item The Euclidean distance from $x$ to $\T$, i.e.\ the distance 
	in $\RR^{n(\nvar+1)}$, 
	is $\sqrt{2(n-t)}$. Specifically, we have
	\begin{multline} \label{E:bunch.of.neg.y0s.is.closest.to.T}
	     \text{ The closest point of } \T \text{ to } x \in \Ss_{t} 
	       \text{ in Euclidean distance is }
		x_{0} := (-y_{0}, \ldots, -y_{0}) \\
		   \text{ and the Euclidean distance from $x$ to $\T$ is } 
		         \sqrt{2(n-t)}.
	\end{multline} 
		\label{I:aug.mean.euc.dist}
	\item The geodesic distance from $\Ss_{t}$ to $\T$ satisfies 
	  \begin{equation}  \label{E:rho.t.approx.sqrt.2(n-t)}
		\rho_{t} := \text{dist} (\Ss_{t}, \T) = 
		  \sqrt{2(n-t)} + O(n-t)^{3/2}, \quad \text{ as } t \uparrow n.
	  \end{equation} 
		\label{I:aug.mean.geod.dist}
          \item $\rho_{t}$ is locally Lipschitz in $t \in (n-1, n)$.  \label{I:rho.t.is.locally.Lip}
          \item Let $R_{t} := dist_{n \nvar-\nvar-1}(\Ss_{t}, \T)$ be the essential 
          $\Hm^{n \nvar-\nvar-1}$-distance from $\Ss_{t}$ to $\T$ \emph{in} 
          the manifold $\D$. Then $R_{t} = \rho_{t}$ 
          Thus, $R_{t} / \sqrt{2(n-t)} \to 1$ as $t \uparrow n$, so $R_{t} \to 0$ 
          as $t \uparrow n$, and $R_{t}$ is locally Lipschitz in $t$. 
                \label{I:aug.mean.essential.dist.approx}
	\item We have
	  \begin{equation}  \label{E:R.n.nvar-nvar-1.H.bdd.above} 
		R_{t}^{-(n \nvar-\nvar-1)} \Hm^{n \nvar-\nvar-1}(\Ss_{t}) 
		       \text{ is bounded above as } t \uparrow n.
	  \end{equation}
($\Hm^{n \nvar-\nvar-1}$ is calculated w.r.t.\ geodesic metric on $\D$.)     \label{I:aug.mean.R.vol.bounded}
  \end{enumerate}
   \end{prop} 

Recall \eqref{E:asymp}. Combining \eqref{E:R.n.nvar-p-1.H.bdd.away.frm.0}, 
\eqref{E:R.n.nvar-nvar-1.H.bdd.above}, and \eqref{E:rho.t.approx.sqrt.2(n-t)}, we have
    \begin{equation} \label{E:aug.mean.sing.vol.asymp}
      \Hm^{n \nvar-\nvar-1}(\Ss_{t}) 
        \asymp \bigl( \sqrt{2(n-t)} \, \bigr)^{n \nvar-\nvar-1}, 
          \text{ as } t \uparrow n .
    \end{equation}
More precisely, let $\Gamma < \infty$ be the upper bound in 
\eqref{E:R.n.nvar-nvar-1.H.bdd.above} and $\gamma > 0$ the lower bound in 
\eqref{E:R.n.nvar-p-1.H.bdd.away.frm.0}. Then
    \begin{equation*}
        \gamma R_{t}^{n \nvar-\nvar-1} \leq \Hm^{n \nvar-\nvar-1}(\Ss_{t}) 
          \leq \Gamma R_{t}^{n \nvar-\nvar-1} .
    \end{equation*}
Hence, if for some $t \in [0,n)$ we could compute $\Hm^{n \nvar-\nvar-1}(\Ss_{t})$ and 
$R_{t}$, we would get an upper bound on the mysterious constant $\gamma$ for the location problem on a sphere (for given $n$ and $\nvar$). In appendix \ref{Chptr:F:CircleSing.data.and.calcs} we actually calculate 
$\Hm^{n \nvar-\nvar-1}(\Ss_{t})$ for two values of $t$ (and $n = 17$, $\nvar = 1$). It remains to compute at least one of the corresponding $R_{t}$'s. 

Figure \ref{F:oneZeroDimP} shows examples for which the factor $R^{d-p-1}$ is achieved. \eqref{E:R.n.nvar-nvar-1.H.bdd.above} gives another.  

For $t \in (0,n)$, define $\rho_{t}$ as in \eqref{E:rho.t.approx.sqrt.2(n-t)}. For $r > 0$, the following gives an approximate solution, $t$, to the equation $\rho_{t} = r$ for $r$ close to $0$. Use the notation in proposition \ref{P:rate.of.decrease.of.Hm.St.aug.direct.mean}. See appendix \ref{Chptr:misc.proofs} for the proof.
  \begin{lemma}  \label{L:solve.rho.t=r}
If $r > 0$ is sufficiently small, there exists $t_{r} \in (0,n)$ s.t.\ $t = t_{r}$ solves 
$R_{t} = \rho_{t} = r$. We have 
    \begin{equation*}
      t _{r} = n - \tfrac{1}{2} r^{2} + O(r^{4}) \text{ as } r \downarrow 0 .
    \end{equation*}
  \end{lemma}

\chapter{Robust Measures of Location on the Circle}  \label{Chptr:robst.loc.on.circle}
\section{Exactness of fit} \label{S:exactness.of.fit}
The ``exact fit property'' seems to be ordinarily defined in the context of regression (subsection \ref{SS:lin.reg.and.LS}; Rousseeuw and Leroy \cite[p.\ 60]{pjRamL03.robust}) but we adapt the idea to measuring location on a sphere. 

Let $x = (y_{1}, \ldots, y_{n}) \in \D := (S^{\nvar})^{n}$, with $y_{i} \in S^{\nvar}$ for $i = 1, \ldots, n$. In this chapter we continue to assume \eqref{E:n>2,nvar>0} and $k$ will be an integer in $[0,n/2)$: 
	\begin{equation}  \label{E:n.nvar.k.sizes}
	   n > 2, \; \nvar > 0, \text{ and } 0 \leq k < n/2,
	\end{equation}
(Since $n > 2$, there is at least one integer $k$ satisfying $0 < k < n/2$.) Recall that each $y_{i}$ is called an ``observation''. 
  \begin{definition}  \label{D:exactness.of.fit}
Say that a measure of location, $\Phi$, on a sphere has ``exactness of fit of order $k$ (with sample size $n$)'' if the following holds. 
If $x = (y_{1}, \ldots, y_{n}) \in \D := (S^{\nvar})^{n}$ with $n-k$ observations $y_{i}$ having a common value $y \in S^{\nvar}$ then $\Phi(x)$  is defined and equals $y$, no matter what the values of the remaining $k$  observations are.
  \end{definition}

(We relax this condition in \eqref{E:exactness.of.fit.loose}.) Note that the augmented mean, $\mu_{a}$, defined in \eqref{E:aug.direct.mean.defn} has order of exact fit 0.  
Moreover, if $\Phi : \D \partlyto S^{\nvar}$ has order of exactness of fit $k$ then it automatically has order of exactness of fit $\ell$ for $\ell = 1, \ldots, k$. In particular, this is true for $\ell = 0$: 
	\begin{multline}  \label{E:exact.fit.then.on.diagnl.exact}
		\text{If } \Phi : \D \partlyto S^{\nvar} 
		  \text{ has exactness of fit of order $k$ for some } 
		  k \in [0, n/2) \\
		    \text{ then $\Phi$ satisfies } \eqref{E:Phi.on.diagnl.exact} .
	\end{multline}

The hope is that a measure of location with order of exact fit $k$ is ``robust'' or ``resistant'' in the sense that it is little affected if as many as $k$ observations are wrong or otherwise unindicative of the ``central tendency'' of the population. More generally, let $\mcl{E} \subset \D$. Say that a map, $\nu : \mcl{E} \to S^{q}$ has ``exactness of fit of order $k$ (with sample size $n$)'' if 
$x = (y_{1}, \ldots, y_{n}) \in \mcl{E}$ and $n-k$ observations have a common value $y$ then $\nu(x)$ is defined and equals $y$ no matter what the values of the remaining $k$ observations are.

Let $\Pf_{k} \subset (S^{\nvar})^{n}$ be the perfect fit space appropriate for measures of location having exactness of fit of order $k$. Specifically,  
    \begin{multline}  \label{E:Pk.defn.robust.loc.on.sphere}
      \Pf_{k} \text{ consists of points } (y_{1}, \ldots, y_{n}) \in (S^{\nvar})^{n} \\
        \text{ s.t.\ at least $n-k$ of the observations $y_{i} \in S^{\nvar}$ are equal.} 
    \end{multline}
Thus, $\Pf_{k}$ is compact and $\Pf_{0} = \T$ defined in \eqref{E:directional.T.defn}. (Do not confuse $\Pf_{k}$ with $\Pf^{k}$ defined in \eqref{E:Perfect.Fits.in.plane.fitting}. And the two instances of $k$ mean something different, too.)

Let $\Phi$ be as in \eqref{E:exact.fit.then.on.diagnl.exact}. Let $\Ss$ be the singular set of $\Phi$ and $\Ss^{\msf{V}_{\pi/2}}$ 
the set of $\msf{V}_{\pi/2}$-severe singularities of $\Phi$. To avoid useless and uninteresting examples, we always assume \eqref{E:assume.Phi.S'.sym} holds and include in the definition of measure of location with order of exactness of fit $k$ the following requirement. 
	\begin{equation}  \label{E:no.severe.sings.in.Pk}
	   \Pf_{k} \setminus \Ss \text{ is dense in } \Pf_{k}  \text{ and }
	     \Ss^{\msf{V}_{\pi/2}} \cap \Pf_{k} = \varnothing .
	\end{equation}
Now, by \eqref{E:SV.is.closed}, $\Ss^{\msf{V}_{\pi/2}}$ is closed. $\Pf_{k}$ is compact. 
Therefore, if $\Phi$ satisfies \eqref{E:no.severe.sings.in.Pk}, then
	\begin{equation}  \label{E:no.severe.sings.near.Pk}
		\text{There is a neighborhood of } \Pf_{k} \text{ containing no } 
		  \msf{V}_{\pi/2} \text{-severe singularities of } \Phi .
	\end{equation}

The following lemma satisfies a requirement
of theorem \ref{T:if.lin.combo.on.F.then.can.rstrct.to.bad.sings} part \ref{I:Omega.Phi.agree.on.Pf} with $\Pf_{k}$ in place of $\Pf$.  
Recall, by \eqref{E:N.sub.n}, $\NN_{n} := \{1, \ldots, n \}$ and recall that $S_{n}$ is the group of permutations of $\NN_{n}$. If $x = (y_{1}, \ldots, y_{n}) \in \D$ write 
$\sigma(x) := (y_{\sigma(1)}, \ldots, y_{\sigma(n)})$ ($\sigma \in S_{n}$). 
Thus, $\Pf_{k}$ is $S_{n}$ invariant.

  \begin{lemma}  \label{L:Pk.is.nhbd.retract}
For $k \in [0, n/2)$, there is a neighborhood $\clU \subset \D$ of $\Pf_{k}$ in $\D$ and 
a retraction $R : \clU \to \Pf_{k}$ onto $\Pf_{k}$ s.t.\ $\sigma (\clU) = \clU$ 
and $R \circ \sigma = \sigma \circ R$ on $\clU$ for every $\sigma \in S_{n}$.
  \end{lemma}
See appendix \ref{Chptr:misc.proofs} for the proof. Also see appendix \ref{Chptr:misc.proofs} for the proof of the following.
  \begin{corly}  \label{C:local.exactness.of.fit}
Let $k \in [0, n/2)$ and let $\clU$ be as in lemma \ref{L:Pk.is.nhbd.retract}. 
Then there exists a continuous measure of location, $\nu_{R} : \clU \to S^{\nvar}$, satisfying \eqref{E:assume.Phi.S'.sym}, with order of exactness of fit $k$.
  \end{corly} 
 
Here we generalize corollary \ref{C:Haus.measure.of.severe.sings.on.spheres} and remark \ref{R:loose.exactness.of.fit.and.homotopy}. (See appendix \ref{Chptr:misc.proofs} for the proof.) Recall the definition, \eqref{E:V.theta.defn}, of $\msf{V}_{\pi/2}$.

  \begin{lemma} \label{L:loose.exactness.of.location.fit}
Suppose $\Phi : \D' \to S^{\nvar}$ satisfies \eqref{E:D'.=.D.less.S}
and \eqref{E:assume.Phi.S'.sym}. Let $\Ss$ be the singular set of $\Phi$. Suppose $\Ss' := \overline{\Ss}$ satisfies \eqref{E:loc.measure.has.no.sings.near.T}. Let $\Ss^{\msf{V}_{\pi/2}}$ be the set of $\msf{V}_{\pi/2}$-severe singularities of $\Phi$ and let $k \in [0, n/2)$.
Suppose $\Phi$ satisfies \eqref{E:no.severe.sings.in.Pk} and can be extended to be defined and continuous on $\Pf_{k}$. (So the restriction $\Phi \restriction_{\Pf_{k}}$ is defined and continuous on $\Pf_{k}$.) Suppose further that for every 
$y, y_{n-k+1}, \ldots, y_{n} \in S^{\nvar}$ we have
	\begin{equation}  \label{E:exactness.of.fit.loose}
             \Phi(y, \ldots, y, y_{n-k+1}, \ldots, y_{n}) \cdot y > -1.
	\end{equation}
Here ``$y, \ldots, y$'' represents $n-k$ copies of $y$. (Allow $k = 0$, in which case 
$(y, \ldots, y, y_{n-k+1}, \ldots, y_{n})$ $:= (y, \ldots, y) \in (S^{\nvar})^{n}$.)  
Then there exists a measure of location on $S^{\nvar}$, symmetric in its arguments, having order of exactness of fit $k$, and continuous 
on $\D \setminus \Ss^{\msf{V}_{\pi/2}}$.

It follows that
    \begin{equation}  \label{E:poz.measure.for.exact.fit.sings}
      \Hm^{n \nvar - \nvar-1}(\Ss^{\msf{V}_{\pi/2}}) > 0 \, \text{  so  } \,
        \text{codim} \, \Ss^{\msf{V}_{\pi/2}} \leq \nvar + 1 .
   \end{equation}
  \end{lemma}

Note that, as observed after \eqref{E:location.Phi.on.diagnl.loose}, \eqref{E:exactness.of.fit.loose} holds if and only if 
$\Phi(y, \ldots, y, y_{n-k+1}, \ldots, y_{n}) + y \neq 0$. 
I.e., $\Phi(y, \ldots, y, y_{n-k+1}, \ldots, y_{n})$ and $y$ are not antipodal.

As mentioned above, the augmented directional mean, $\mu_{a}$ ($a \in [0, n)$) has order of exactness of fit 0. Next, we consider whether the singular set of $\mu_{a}$ might  include that of a measure of location with \emph{positive} order of exactness of fit. For the proof of the following see appendix \ref{Chptr:misc.proofs}.

  \begin{prop} \label{P:aug.mean.sing.set.like.exact.fit}
Let $k \in [0, n/2)$. If $a \in [0, n-2k)$ then $\mu_{a}$ satisfies \eqref{E:no.severe.sings.near.Pk} and there exists a measure of location on $S^{\nvar}$ with order of exactness of fit $k$ whose singularities are all 
$\msf{V}_{\pi/2}$-severe singularities of $\mu_{a}$. But if $a \in [n-2k, n)$, 
then $\mu_{a}$ has $\msf{V}_{\pi/2}$-severe singularities in $\Pf_{k}$ and so violates \eqref{E:no.severe.sings.near.Pk}.
  \end{prop}
By \eqref{E:Sa=Sa'}, $\Ss_{a}^{\msf{V}_{\pi/2}}$ is the set of \emph{all} singularities 
of $\mu_{a}$. Thus, if $a \in [0, n-2k)$ one has the option of gaining resistance to outliers by replacing $\mu_{a}$ by a measure of location on $S^{\nvar}$ with order of exactness of fit $k$ while paying no extra cost in terms of its singular set. 

\section{Augmented directional median}   \label{SS:AugDirectMedian}
Lemma \ref{L:loose.exactness.of.location.fit} and proposition \ref{P:aug.mean.sing.set.like.exact.fit} assert the existence of measures of location on the sphere with positive order of exactness of fit. Here we construct an explicit family of examples of such.

For the rest of this chapter we focus on data on a circle. I.e., 
    \begin{equation}  \label{E:ma:nvar=1}
      \nvar=1 \text{ so } \D := (S^{1})^{n} .
    \end{equation}
The ``spherical median'' (Fisher \emph{et al} \cite[p.\ 111]{niFtLbjjE87};  Fisher \cite{niF85.SphericalMedians}) is defined as follows. Recall the definition of angle, \eqref{E:angle.between.vectors}. Given a data set $x = (y_{1}, \ldots, y_{n}) \in (S^{1})^{n}$, the spherical median (perhaps in the $\nvar =1$ case we should call it the ``directional median'') is the point $v = m(x) \in S^{1}$ that minimizes.
	\begin{equation*}
		G(v;x) := \sum_{i=1}^{n} \angle (y_{i}, v),  \quad v \in S^{1},
	\end{equation*}
whenever the minimization has a unique solution. 

In analogy with the augmented directional mean, \eqref{E:aug.direct.mean.defn}, define the ``augmented directional median'' as follows.  Let $a > 0$ and let $y_{0} \in S^{1}$ be fixed. Call $a$ the ``augmentation weight'' and $y_{0}$ the ``augmentation point''. 

Then the ``augmented directional median'' of $x$ is the point $v = m_{a}(x) \in S^{1}$ that minimizes
	\begin{equation}  \label{E:Ga.defn}
		G_{a}(v; x) := G_{a}(v) := a \angle(v, y_{0}) +  \sum_{i=1}^{n} \angle (y_{i}, v),  
		            \quad v \in S^{1},
	\end{equation}
whenever the minimization has a unique solution. (See lemma \ref{L:data.maps.defined.by.opt}.) 
Note that $G_{a}(v;x)$ is continuous -- by compactness, \emph{uniformly} continuous -- in $(v;x) \in S^{1} \times \D$. Moreover, $\Phi = m_{a}$ obviously has property \eqref{E:assume.Phi.S'.sym}.

Let $k$ be a positive integer $< n/2$. We want $m_{a}$ to have order of exact fit $k$. (See definition \ref{D:exactness.of.fit}.) Suppose $a \geq n-2k$. 
Let $y_{i} = y \in S^{1}$ for $i = 1, \ldots, n-k$ and let $y_{i} = y_{0}$ for $i = n-k+1, \ldots, n$. Then
    \begin{equation*}
        G_{a}(v; x) = (a + k) \angle(v, y_{0}) + (n-k) \angle (y, v).
    \end{equation*}
Hence, $G_{a}(y; x) = (a + k) \angle(y, y_{0}) \geq (n-k) \angle(y, y_{0})$ 
and $G_{a}(y_{0}; x) = (n-k) \angle(y, y_{0})$. Then $G_{a}(v; x)$ is not uniquely minimized by $v = y$, if at all. Thus, for these values of $a$ and $k$, $m_{a}$ does \emph{not} have order of exact fit, $k$.

Continue to assume \eqref{E:n.nvar.k.sizes} holds (\emph{viz.}\ $n > 2$) but now assume $0 < a < n-2k$. In summary, we assume
	\begin{equation}  \label{E:ma.n.k.a.sizes}
	n > 2, \; 0 < k < n/2, \text{ and } 0 < a < n-2k.
	\end{equation}
Exactness of fit of $m_{a}$ is investigated in section \ref{S:exactness.of.m_a}. 
\emph{We will assume}
	\begin{equation}   \label{E:.for.aug.median.a.not.integer}
		a \text{ is not an integer.}
	\end{equation}
In fact, in section \ref{SS:dist.from.anglar.median.to.sev.sings.to.Pk} we are most interested 
in $a \in (n - 2k -1, n-2k)$. The fact that $y_{0}$ has non-integer ``multiplicity'' distinguishes $y_{0}$ from the \emph{observations} $y_{1}, \ldots, y_{n}$. 

In general, if $v \in S^{1}$ minimizes $G_{a}( \cdot;x)$, then $G_{a}(v;x) \leq G_{a}(-v;x)$. From \eqref{E:angle.-x,y} it follows that 
	\begin{equation*}
		(a + n)^{-1} \left( a \angle \bigl( m_{a}(x), y_{0} \bigr) 
		          + \sum_{i=1}^n{} \angle \bigl( m_{a}(x), y_{i} \bigr) \right) < \pi/2.
	\end{equation*}

\section{Construct a dense set, $\D' \subset \D := (S^{1})^{n}$ on which $m_{a}$ is continuous} 
    \label{ SS:construct.dense.circle.set}
Once we construct a $\D'$ satisfying \eqref{E:D'.dense.Phi.cont.on.D'}, lemma \ref{L:extend.Phi.to.D.less.S} can be applied so that, enlarging $\D'$ if necessary, \eqref{E:D'.=.D.less.S} holds. Let $v \in S^{1}$ and let 
$\phi := \phi_{v} : (-\pi, \pi] \to S^{1}$ parametrize $S^{1}$ by arc length from $v$ 
so $\phi(0) = v$. Specifically, given $s,t \in (-\pi, \pi]$, $|s-t|$ is the length of one of the two shortest arcs connecting $\phi(s)$ and $\phi(t)$. (There are arcs joining $\phi(s)$ 
and $\phi(t)$ that wrap around $S^{1}$ arbitrarily many times. We say that $\phi$ is a parametrization ``at $v$".) Now, $\angle$ is the metric on $S^{1}$ defined to be the length of the shorter arc joining two points. Thus, if $s, t \in (-\pi, \pi]$, we have
    \begin{equation}  \label{E:angle.in.terms.of.phi}
      \angle \bigl[ \phi(s), \phi(t) \bigr] = \min \bigl\{ |s-t|, 2 \pi - |s-t| \bigr\}.
    \end{equation}

Let ``$\cdot$'' be the usual inner product on $\RR^{2}$. Now let $v' \in S^{1}$ with $v' \notin \{ v, -v \}$ so $v \cdot v' \in (-1, 1)$. We say that $\phi$ ``turns toward'' $v'$ if $v' = \phi_{v}(s)$ with 
$0 < s \leq \pi$. Thus, as $u \in (0, s]$ decreases to 0, $\angle \bigl[ \phi(u), v' \bigr]$ increases. $\phi_{v}$ ``turns away'' from $v'$ is the opposite: If $\phi_{v}$ turns toward $v'$ then $s \mapsto \phi_{v}(-s)$ turns away. If $f$ is a function on $S^{1}$ we say that some behavior of $f$ pertains to ``turning toward'' $v'$ if it refers to $f \circ \phi_{v}$ for a parametrization turning toward $v'$. ''Turning away from'' $v'$ has the opposite meaning.

For example, let $\tfrac{d_{+}}{dt} \restriction_{t=s}$ denote right derivative, i.e. the limit of the difference quotient as $t \downarrow s$, and $\tfrac{d_{-}}{dt}$ denote left derivative.
Suppose $\phi$ turns toward $v'$. Then the derivative of $f$ at $v$ ``turning toward'' $v'$ 
is $\tfrac{d_{-}}{dt} f \circ \phi(t) \restriction_{t=0}$, providing it exists. And we write  
$\tfrac{d_{-}}{dy'} f(y') \restriction_{y'=v} = \tfrac{d_{-}}{dt} f \circ \phi(t) \restriction_{t=0}$, where $v = \phi(0)$. The derivative of $f$ at $v$ ``turning away'' from 
$v'$ is $\tfrac{d_{+}}{dt} f \circ \phi(t) \restriction_{t=0}$, providing it exists, and we define 
$\tfrac{d_{+}}{dy} f(y') \restriction_{y'=v} = \tfrac{d_{+}}{dt} f \circ \phi(t) \restriction_{t=0}$. Of course, if $f$ is differentiable at $v$ the two derivatives will be equal 
and we define $\tfrac{d_{}}{dy} f(y') \restriction_{y'=v}$ to be their common value. One-sided partial derivatives are defined similarly. We say that all these derivatives, one-sided or not, partial or not, are taken ``along $S^{1}$''.

Let $w \in S^{1}$ be orthogonal to $v$, i.e., $v \cdot w = 0$. Let $\phi$ be a parametrization at $v$ turning toward $w$. So $w = \phi(\pi/2)$. 
Suppose $s \in (-\pi, \pi)$, so $s \neq \pi$. Suppose also that 
$s \neq 0$. Then, for $t$ satisfying $|t| < \min \bigl\{ |s|, \pi - |s| \bigr\}$, 
we have $|s-t| \leq |s| +|t| < |s| + \pi - |s|$ so, by \eqref{E:angle.in.terms.of.phi}, 
$\angle \bigl[ \phi(s), \phi(t) \bigr] = |s-t|$. If $s < 0$, i.e., $\phi(s) \cdot w < 0$, 
then $t-s = t + |s| > -|s| + |s| = 0$. Let $sign \, u \in \{-1, 0, 1 \}$ be the sign 
of $u \in \RR$. (See \eqref{E:sign.function}.)
Therefore, $\angle \bigl[ \phi(s), \phi(t) \bigr] = t-s$, 
so $\tfrac{d}{dt} \angle \bigl( \phi(s),\phi(t) \bigr) \restriction_{t=0} = 1 =  - sign \, (\phi(s) \cdot w)$. 
If $s > 0$, i.e., $\phi(s) \cdot w > 0$, then again with $|t| < \min \bigl\{ |s|, \pi - |s| \bigr\}$, we have
$s - t = |s| - t > |s| - |s| = 0$. So $\angle \bigl[ \phi(s), \phi(t) \bigr] = s - t$ and 
$\tfrac{d}{dt} \angle \bigl( \phi(s),\phi(t) \bigr) \restriction_{t=0} = -1  
= - sign \, (\phi(s) \cdot w)$.

In summary, suppose $v \in S^{1}$, $\phi(0) = v$, $\phi : (-\pi, \pi] \to S^{1}$, \eqref{E:angle.in.terms.of.phi} holds, and $w = \phi(\pi/2)$. Then, taking $y = \phi(s)$, we have
	\begin{multline}  \label{E:angle.deriv.sign}
		\text{If } y \in S^{1} \setminus \{ v, -v \} \text{ then } 
		  \frac{d}{dt} \angle \bigl( y,\phi(t) \bigr) \restriction_{t=0} \text{ exists and } \\
		    \frac{d}{dt} \angle \bigl( y,\phi(t) \bigr) \restriction_{t=0} = -sign(y \cdot w).
	\end{multline}  

For $s \in (-\pi, \pi]$, $\angle \bigl( v,\phi(s) \bigr) = |s|$ 
and, by \eqref{E:angle.-x,y}, $\angle \bigl( -v,\phi(s) \bigr) = \pi - |s|$.
Thus, if $s \in \{ 0, \pi \}$, then $\tfrac{d}{dt} \angle \bigl( \phi(s),\phi(t) \bigr) \restriction_{t=0}$ is not defined, but the one-sided derivatives are defined:  
    \begin{align} \label{E:one-sided.angle.derivs}
      \frac{d_{+}}{dt} \angle \bigl( v,\phi(t) \bigr) \restriction_{t=0} 
        &= \frac{d_{-}}{dt} \angle \bigl( -v,\phi(t) \bigr) \restriction_{t=0} = 1 \text{ and } \\
      \frac{d_{-}}{dt} \angle \bigl( v,\phi(t) \bigr) \restriction_{t=0} 
        &= \frac{d_{+}}{dt} \angle \bigl( -v,\phi(t) \bigr) \restriction_{t=0} = -1 . \notag
    \end{align}
Write $u_{0} := \phi(0)$. We write
    \begin{equation}  \label{E:(d/du).angle}
      \frac{d}{du} \angle (y, u) \restriction_{u=u_{0}} 
        = \frac{d}{dt} \angle \bigl( y,\phi(t) \bigr) \restriction_{t=0}.
    \end{equation}
Similarly for $\frac{d+}{du} \angle (y, u) \restriction_{u=u_{0}}$ and 
$\frac{d-}{du} \angle (y, u) \restriction_{u=u_{0}}$.

Given $x \in \D$, let
	\begin{multline} \label{E:Y(x).defn}
	     (y_{1}, \ldots, y _{n}) := x, \;
		  Y := Y(x)  := \{ y_{0}, y_{1}, \ldots, y_{n} \} \subset S^{1} \\
		           \text{ and } - Y := -Y(x) := \{ -y_{0}, -y_{1}, \ldots, -y_{n} \} \subset S^{1}.
	\end{multline}
So $Y(x), -Y(x) \subset S^{1}$ are the \emph{sets} of unique locations 
of $y_{0}, y_{1}, \ldots, y_{n}  \in S^{1}$ and $-y_{0}, -y_{1}, \ldots, -y_{n} \in S^{1}$, resp., and the cardinalities $\bigl| Y(x) \bigr|$, $\bigl| -Y(x) \bigr|$, though equal, can be less than $n$. Thus, $Y$ and $-Y$ are not samples (remark \ref{R:sample}). By contrast, $x$ is an $n$-\emph{tuple} of points in $S^{1}$ In fact, by \eqref{E:assume.Phi.S'.sym}, in  this discussion the order of the $y_{i}$'s in $x$ does not matter, $x$ \emph{is} a sample. Define
	\begin{equation} \label{E:Ma.defn}
		M_{a}(x) \subset S^{1} = \text{ the set of } v \in S^{1} \text{ at which } 
		  v \mapsto G_{a}(v; x) \text{ achieves its minimum.}
	\end{equation} 
(By compactness of $S^{1} \times \D$ and continuity of $G$, we have 
$M_{a}(x) \neq \varnothing$.) 

Recall \eqref{E:.for.aug.median.a.not.integer}. \emph{Claim:} 
	\begin{equation} \label{E:Ma(x).in.Y}
		M_{a}(x) \subset Y(x).
	\end{equation}
Let $v \in M_{a}(x)$. We show $v \in Y$. First, suppose that $v \notin Y \cup (-Y)$ and let $w \in S^{1}$ be perpendicular to $v$. Let $\phi(t)$ be a parametrization at $v$ (so $\phi(0) = v$) turning toward $w$. Since $\{v, -v\} \cap \bigl[ Y \cup (-Y) \bigr] = \varnothing$, by \eqref{E:angle.deriv.sign}, the derivative of $G_{a}$ at $v$ along the circle turning toward $w$ exists and is just
	\[
		\frac{d}{dt} G \bigl[ \phi(t); x \bigr] \restriction_{t=0}
		    = - \left( a \, sign \, (y_{0} \cdot w) 
		          + \sum_{i=1}^{n} sign \, (y_{i} \cdot w) \right).
	\]
(See \eqref{E:sign.function}.) Since $v \notin Y \cup (-Y)$ by assumption, we have that 
$sign \, (y_{i} \cdot w)$ ($i=0, \ldots, n$) are all nonzero (because the only points 
of $S^{1}$ orthogonal to $w$ are $v$ and $-v$). Hence, by \eqref{E:.for.aug.median.a.not.integer}, this derivative cannot be 0. Hence, by moving in one direction or the other $G_{a}$ can be decreased. This contradicts $v \in M_{a}(x)$ and proves that 
    \begin{equation}  \label{E:v.in.Y.union.(-Y)}
      v \in Y \cup (-Y) .
    \end{equation}

Now suppose $G_{a}(\cdot; x)$ is minimized (not necessarily uniquely) by $v \in Y \cup (-Y)$ 
and again let $w \in S^{1}$ be perpendicular to $v$. Let $c = 0, 1, 2, \ldots, n$ be the number of points $y_{i}$ ($i = 1, \ldots, n$; we exclude $y_{0}$) for which $w \cdot y_{i} > 0$ and 
let $d = 0, 1, 2, \ldots, n$ be the number of points $y_{i}$ ($i = 1, \ldots, n$) 
for which $w \cdot y_{i} < 0$. Let $e = 0, 1, 2, \ldots, n$ be the number of points $y_{i}$ 
($i \in \NN_{n}$) which equal $v$ and let $f = 0, 1, 2, \ldots, n$ be the number of points $y_{i}$ 
($i \in \NN_{n}$) which equal $-v$. (For such $i$, $y_{i} \cdot w = 0$.) Thus, $c + d + e + f = n$. To sum up:
    \begin{align*}
        c &:= \text{ number of $i > 0$ s.t.\ } w \cdot y_{i} > 0, \\
        d &:= \text{ number of $i > 0$ s.t.\ } w \cdot y_{i} < 0, \\
        e &:= \text{ number of $i > 0$ s.t.\ } y_{i} = v, \text{ and } \\
        f &:= \text{ number of $i > 0$ s.t.\ } y_{i} = -v.
    \end{align*}

Suppose first that $v \neq \pm y_{0}$. WLOG $w \cdot y_{0} > 0$. (Otherwise, replace $w$ by $-w$.) Then, by \eqref{E:angle.deriv.sign} and \eqref{E:one-sided.angle.derivs}, the one-sided derivative 
of $G_{a}(\cdot; x)$ at $v$ turning toward $w$, i.e., 
$\tfrac{d_{-}}{dt} G \bigl( \phi(s), x \bigr) \restriction_{s=0}$ for $y \in Y$,  is $-e + f - c + d - a$. 
Since $v' = v = \phi(0)$ minimizes $G_{a}(v', x)$, we have $G_{a} \bigl[ \phi(s), x \bigr]$ decreases as $s$ increases to $0$. Hence,
$-e + f - c + d - a = \tfrac{d_{-}}{dt} G \bigl( \phi(s), x \bigr) \restriction_{s=0} \leq 0$. 
Thus, if $v \neq \pm y_{0}$,
        \begin{equation}  \label{E:left.deriv.acdef}
          e - f + c - d + a \geq 0 .
        \end{equation} 

Similarly, since $v' = v = \phi(0)$ minimizes $G_{a}(v', x)$, we have $G_{a} \bigl[ \phi(s), x \bigr]$ decreases as $s$ decreases to $0$. But as $s \downarrow 0$ the point $\phi(s)$ turns away from $w$. Therefore, the one-sided derivative, $\tfrac{d_{+}}{dt} f \circ \phi(t) \restriction_{t=0}$, 
of $G_{a}(\cdot; x)$ at $v$ turning away from $w$ is non-negative. By \eqref{E:angle.deriv.sign} and \eqref{E:one-sided.angle.derivs}, that derivative 
is $e - f - c + d - a$. Thus, if $v \neq \pm y_{0}$,
    \begin{equation}  \label{E:right.deriv.acdef}
      e - f - c + d - a \geq 0 .
    \end{equation}  
Combining \eqref{E:left.deriv.acdef} and \eqref{E:right.deriv.acdef} we conclude,
	\begin{equation}  \label{E:a.c.d.e.f.ineq}
		e - f \geq 0 \text{ and } n-c = d+e+f \geq e-f+d \geq a + c, 
		  \text{ if } v \neq \pm y_{0}. 
	\end{equation}

Suppose $v \in (-Y) \setminus Y$, but $v \neq -y_{0}$. Then there must be some $i > 0$ s.t.\ 
$v = -y_{i}$. Thus, $f > 0$. 
On the other hand, $v \notin Y$ so $e = 0$, Therefore, $e-f < 0$, contradicting \eqref{E:a.c.d.e.f.ineq}. 
By \eqref{E:v.in.Y.union.(-Y)}, this proves that $v \in Y$ if $v \neq \pm y_{0}$.

If $v = y_{0}$ then $v \in Y$ and we are done. So suppose $v = - y_{0}$. If $-y_{0} \in Y$ then $v \in Y$ and we are again done. So suppose $v = -y_{0} \notin Y$. Thus, $e = 0$ and $y_{0} = -v$. Using \eqref{E:angle.deriv.sign} and \eqref{E:one-sided.angle.derivs} again, we see that 
the one-sided derivative of $G_{a}(\cdot; x)$ at $v$ turning toward $w$ is $f-c+d+a$. As before, this is non-positive, so $-f+c-d-a \geq 0$. Again as before, the one-sided derivative 
of $G_{a}(\cdot; x)$ at $v$ turning away from $w$ is non-negative. That derivative is $-f-c+d-a$.
I.e., $-f-c+d-a \geq 0$. Hence, we get the following. 
	\begin{equation}  \label{E:e-a-f.geq.0}
	         -f+c-d-a \geq 0 \text{ and } -f-c+d-a \geq 0 , \text{ so } a + f \leq 0, 
	           \text{ if } v = - y_{0} \notin Y.
	\end{equation} 
But $f \geq 0$ and, by \eqref{E:ma.n.k.a.sizes}, $a > 0$. 
This contradicts \eqref{E:e-a-f.geq.0}.
This proves the claim \eqref{E:Ma(x).in.Y} that $M_{a}(x) \subset Y$. 

Let $y \in S^{1}$ and $x = (y, \ldots, y) \in \T$. (See \eqref{E:directional.T.defn}.) By \eqref{E:Ma(x).in.Y}, $M_{a} \subset \{ y_{0}, y \}$. Thus, by \eqref{E:Ga.defn} and \eqref{E:ma.n.k.a.sizes}, we have 
$G_{a}(y;x) = a \angle(y_{0}, y) < (n-2) \angle(y_{0}, y) < n \angle(y_{0}, y) 
= G_{a}(y_{0}; x)$. Hence, $m_{a}(x) = y$. I.e., 
    \begin{equation}  \label{E:ma.defined.on.T}
      m_{a} \text{ is defined on } \T \text{  and satisfies \eqref{E:Phi.on.diagnl.exact}.}
    \end{equation}

Let $z_{1}, \ldots, z_{t} \in S^{1}$ be the distinct locations of the points 
$y_{1}, \ldots, y_{n}$. Call $\{ z_{1}, \ldots, z_{t} \}$ the ``support'' of $x$. We only count observations, i.e.\ $y_{1}, \ldots, y_{n}$, not the augmentation point $y_{0}$, but it is possible that $y_{0} \in \{ z_{1}, \ldots, z_{t} \}$. 
If $y_{0} \notin \{ z_{1}, \ldots, z_{t} \}$, define $z_{0} = y_{0}$. 
Let $\ell_{\alpha} \in [1,n]$ be the multiplicity of $z_{\alpha}$, i.e., $\ell_{\alpha}$ is the number of $y_{i}$'s that equal $z_{\alpha}$ ($\alpha=1, \ldots, t$). 
So $\ell_{1} + \cdots + \ell_{t} = n$. We \emph{claim:}
  \begin{lemma} \label{L:not.both.antipodals.in.Ma}
If $z_{\beta} = -z_{\alpha}$ ($\alpha, \beta = 0, \ldots, t$) then at most 
one of $z_{\alpha} $ and $z_{\beta}$ is in $M_{a}(x)$.
  \end{lemma}
For proof see appendix \ref{Chptr:misc.proofs}.
 
Let $\D' \subset \D$ be the set 
    \begin{equation}   \label{E:ma.D'.defn}
         \D' := \bigl\{ x \in \D' : m_{a}(x) \text{ is defined } \bigr\} .
    \end{equation}
Thus, $x \in \D'$ if and only if $G_{a}(\cdot ; x)$ has exactly one global minimum, 
in $Y(x)$ by \eqref{E:Ma(x).in.Y}. Since $\F = S^{1}$ is compact and as we just observed $m_{a}$ is defined on $\T$ so $\D'$ is not empty, we may apply lemma \ref{L:data.maps.defined.by.opt} part \ref{I:cmpct.opt} to $g = G_{a}$ to conclude that 
$m_{a}$ is continuous on $\D'$. We have the following. For proof of the following see appendix \ref{Chptr:misc.proofs}.

  \begin{lemma} \label{L:D'.is.dense.for.aug.median}
$\D'$ is dense in $\D$ and every point in $(\D')^{c} := \D \setminus \D'$ is a singularity of $m_{a}$. Thus, $(\Phi, \D')$ with $\Phi = m_{a}$ satisfies \eqref{E:D'.=.D.less.S}.
  \end{lemma}
In particular, 
    \begin{equation}  \label{E:ma.satisfies.hyps.of.loose.lemma}
      m_{a} 
        \text{ satisfies the hypotheses of lemma \ref{L:loose.exactness.of.location.fit}.}
    \end{equation} 
(So it satisfies the conclusion as well!)

\section{Severe singularities of $m_{a}$} \label{SS:severe.sings.ma} 
Let $x$ be a $\msf{V}_{\pi/2}$-severe singularity of $m_{a}$. (See \eqref{E:V.theta.defn}.) We \emph{claim} that $M_{a}(x)$ does not lie in any open semi-circle in $S^{1}$. (See \eqref{E:Ma.defn}. A semi-circle is an arc of length $\pi$.) For suppose $H \subset S^{1}$ is an open semi-circle and $M_{a} \subset H$. Write $x = (y_{1}, \ldots, y_{n})$. 
By \eqref{E:Ma(x).in.Y}, we know that 
$M_{a}(x) \subset Y = \{ y_{0}, y_{1}, \ldots, y_{n} \}$. Suppose $y_{j} \in M_{a}(x)$. 
($j=0$ is possible.)  Let   
	\begin{equation} \label{E:alpha.y1.x}
		g := g(a) := G_{a}(y_{j}; x) = \min_{v \in S^{1}} G_{a}(v; x).
	\end{equation}
By compactness, continuity of $G_{a}(\cdot; x)$, and the fact that 
$\overline{M_{a}(x)} = M_{a}(x) \subset H$, 
we may pick $\epsilon > 0$ s.t.\ 
	\begin{equation}   \label{E:Ga.Hc.3eps.alpha}
		G_{a}(v;x) > g + 3 \epsilon,
			\quad \text{ for every } v \in H^{c},
	\end{equation}
where $H^{c} := S_{1} \setminus H$. 

Let $u \in S^{1}$ be the midpoint of $H$ so $H = V_{u, \pi/2}$ by \eqref{E:sphere.cap.defn}. Since $x$ is a $\msf{V}_{\pi/2}$-severe singularity of $m_{a}$, it follows from definition \ref{D:V-severe} and compactness of $\F = S^{1}$ that there exists a sequence 
$\{ x_{\nu} \} \subset \D'$ s.t.\ $x_{\nu} \to x$ and  
$m_{a}(x_{\nu}) \to  v_{\infty} \notin H$. 
Thus, for every open neighborhood, $\mcl{W}$, of $H^{c}$, eventually $m_{a}(x_{\nu}) \in \mcl{W}$.) By \eqref{E:Ga.Hc.3eps.alpha}, continuity of $G_{a}$, and \eqref{E:alpha.y1.x}, 
	\[
		g + 3 \epsilon \leq G_{a}(v_{\infty}; x) 
		     = \lim_{\nu \to \infty} G_{a} \bigl( m_{a}(x_{\nu}); x_{\nu} \bigr)
		       \leq \lim_{\nu \to \infty} G_{a} \bigl( y_{j}; x_{\nu} \bigr) 
		         = G_{a}(y_{j}; x) = g.
	\]
This contradiction establishes the claim that $M_{a}(x)$ does not lie in any open semi-circle. 

In particular, by \eqref{E:Ma(x).in.Y},
    \begin{equation} \label{E:Y.not.in.semi-circle}
      \text{If } Y(x) \text{ lies in an open semi-circle then 
        $x$ is not a } \msf{V}_{\pi/2}\text{-severe singularity of } m_{a}. 
    \end{equation}

\section{Exactness of fit property of $m_{a}$}  \label{S:exactness.of.m_a}
Let $k$, an integer, and $a$ be as in \eqref{E:ma.n.k.a.sizes}. We show that $m_{a}$ has exactness of fit of order $k$ (definition \ref{D:exactness.of.fit}). Thus, we show first that, 
if $x = (y_{1}, \ldots, y_{n}) \in \D := (S^{1})^{n}$, $0 < i_{1} < \cdots < i_{n-k} \leq n$, 
and $y_{i_{1}} = \cdots = y_{i_{n-k}}$ then $m_{a}(x) =  y_{i_{1}} = \cdots = y_{i_{n-k}}$. Recall the definition \eqref{E:Pk.defn.robust.loc.on.sphere} of $\Pf_{k}$. 
$\Pf_{k}$ will function as the ``perfect fit space'' for $m_{a}$. Note that 
    \begin{equation}  \label{E:dim.Pk.=.k+1}
      \dim \Pf_{k} = k+1 . 
    \end{equation}
We also show that $m_{a}$ has no $\msf{V}_{\pi/2}$-severe singularities in a neighborhood of $\Pf_{k}$. (See \eqref{E:V.theta.defn} for definition of $\msf{V}_{\pi/2}$.) In fact, the same neighborhood of $\Pf_{k}$ works for all $a$ not too close to $n-2k$. 

Let $a_{0} \in [0, n-2k)$ and $\alpha \in (0, \pi/4)$ be fixed and let 
	\begin{equation}  \label{E:epsilon.alpha}
		\epsilon_{0} := \epsilon_{0}(a_{0}, \alpha) := (n-2k-a_{0}) \alpha , 
	\end{equation}
so $\epsilon_{0} > 0$, and let $\epsilon \in (0, \epsilon_{0}]$. Recall, by \eqref{E:N.sub.n}, 
$\NN_{n} := \{1, \ldots, n \}$. If $I \subset \NN_{n}$ has cardinality $n-k$, let
    \begin{multline}  \label{E:U.epsilon.defn}
      \clU_{I}(\epsilon) := \bigl\{ (y_{1}, \ldots, y_{n}) \in \D : |y_{i} - y_{j}| 
        < \epsilon/(n-k-1) \text{ for every } i, j \in I \bigr\} \\
          \text{ and } \clU := \clU(\epsilon) := \bigcup_{I} \clU_{I} ,
    \end{multline}
where the union is over all $I \subset \NN_{n}$ with cardinality $n-k$. So $\clU$ is an open neighborhood of $\Pf_{k}$. Let
    \begin{equation}  \label{E:a.leq.a0}
      a \in (0, a_{0}] \text{ not be an integer; otherwise $a$ is arbitrary.}
    \end{equation}
Then we \emph{claim:}
   \begin{multline}   \label{E:ma.has.no.svr.sings.near.Pk}
	\text{For any non-integer } a \in (0, a_{0}], 
          m_{a} \text{ has no } \msf{V}_{\pi/2} 
            \text{-severe singularities in } \clU(\epsilon_{0}) \\
	      \text{ and no singularities in } \Pf_{k} \text{ of any severity} .
   \end{multline}

Let 
    \begin{equation*}
      x = (y_{1}, \ldots, y_{n}) \in \clU(\epsilon) .
    \end{equation*}
For definiteness, 
suppose $x \in \clU_{\{1, \ldots, n-k \}}(\epsilon)$. Then
	\begin{equation}  \label{E:angle.sum.<.eps}
		\sum_{i=1}^{n-k} \angle(y_{i}, y_{1}) = \sum_{i=2}^{n-k} \angle(y_{i}, y_{1}) 
		            <  \epsilon.
	\end{equation}
Let
	\begin{equation}    \label{E:C.angle.defn}
		C :=  \sum_{i=n-k+1}^{n} \angle(y_{1}, y_{i}).
	\end{equation}
So the $i$'s in the preceding summation and that in \eqref{E:angle.sum.<.eps} are disjoint. Let $v \in S^{1}$. Since $\angle$ is a metric on $S^{1}$, we have, by the triangle inequality,
	\begin{align}   \label{E:Ga.v.leq}
	   G_{a}(v) &= a \angle(v, y_{0}) + \sum_{i=1}^{n} \angle(v, y_{i}) \notag  \\
		         &\leq a \angle(v, y_{0}) + \sum_{i=1}^{n-k} \angle(v, y_{i}) 
		           + \left( k \, \angle(v, y_{1}) + \sum_{j=n-k+1}^{n} \angle(y_{1}, y_{j}) \right). 
	\end{align}
Substituting $y_{1}$ in place of $v$ in the preceding and applying \eqref{E:angle.sum.<.eps}, we get in particular,
	\begin{equation}  \label{E:Ga.x1.leq}
		G_{a}(y_{1}) \leq a \angle(y_{1}, y_{0}) +  \epsilon + C.
	\end{equation}

Let $v_{a} \in M_{a}(x)$ be an arbitrary minimizer of $G_{a}$ in \eqref{E:Ga.defn}. (See \eqref{E:Ma.defn}.) We have the following. 
	\begin{equation}  \label{E:circ.triang.ineqs1}
		\angle(v_{a}, y_{i}) + \angle(y_{i}, y_{1}) \geq \angle(v_{a}, y_{1}), 
		  \quad (i=0, \ldots, n) ,
	\end{equation}
and
	\begin{equation}	 \label{E:circ.triang.ineqs4}
		\angle(v_{a}, y_{j}) + \angle(v_{a}, y_{1}) \geq \angle(y_{j}, y_{1}), 
		  \quad (j=0, \ldots, n).
	\end{equation}
Note that \eqref{E:circ.triang.ineqs4} implies
	\begin{equation}	 \label{E:circ.triang.ineqs5}
		\angle(v_{a}, y_{1}) \geq \angle(y_{j}, y_{1}) - \angle(v_{a}, y_{j}), 
		  \quad (j=0, \ldots, n).
	\end{equation}

From \eqref{E:Ga.x1.leq}, \eqref{E:circ.triang.ineqs1}, and \eqref{E:angle.sum.<.eps},
	\begin{align}  \label{E:a.angle.geq.-eps}
		a \angle(y_{1}, y_{0}) &+  \epsilon + C \notag \\ 
				   &\geq G_{a}(y_{1}) \geq G_{a}(v_{a}) \notag \\
				   &= a \angle(v_{a}, y_{0}) + \sum_{i=1}^{n-k} \angle(v_{a}, y_{i}) 
				          + \sum_{j=n-k+1}^{n} \angle(v_{a}, y_{j})  \\
				   &\geq a \angle(v_{a}, y_{0}) 
				        + \left[ (n-k) \angle(v_{a}, y_{1}) 
				                            - \sum_{i=1}^{n-k} \angle(y_{1}, y_{i}) \right]
				          + \sum_{j=n-k+1}^{n} \angle(v_{a}, y_{j}) \notag \\
				   &\geq a \angle(v_{a}, y_{0}) 
				        + (n-k) \angle(v_{a}, y_{1}) - \epsilon 
				                  + \sum_{j=n-k+1}^{n} \angle(v_{a}, y_{j}).    \notag
	\end{align}

Therefore, by \eqref{E:a.angle.geq.-eps}, \eqref{E:C.angle.defn}, \eqref{E:circ.triang.ineqs4}, and \eqref{E:circ.triang.ineqs5},
	\begin{multline*}  \label{E:n-k.angle.leq}
		 (n-k) \angle(v_{a}, y_{1}) \\
			\begin{aligned}
				{} 
				   &\leq 2 \epsilon 
				     + a \bigl[ \angle(y_{1}, y_{0}) - \angle(v_{a}, y_{0}) \bigr]  
				       + C - \sum_{j=n-k+1}^{n} \angle(v_{a}, y_{j}) \\
				    &= 2 \epsilon + a \bigl[ \angle(y_{1}, y_{0}) - \angle(v_{a}, y_{0}) \bigr]  
				       + \left[  \sum_{j=n-k+1}^{n} \angle(y_{1}, y_{j}) 
				                        - \sum_{j=n-k+1}^{n} \angle(v_{a}, y_{j}) \right] \\   
				    &\leq     2 \epsilon + a \angle(v_{a}, y_{1}) 
				       + k \angle(y_{1}, v_{a}). 
			\end{aligned}
	\end{multline*}  
Rearranging, we get
	\begin{equation*}
		(n-k - a - k) \angle(v_{a}, y_{1}) \leq 2 \epsilon.
	\end{equation*}
By \eqref{E:ma.n.k.a.sizes}, $n-2k - a > 0$. Hence, by \eqref{E:a.leq.a0} and \eqref{E:epsilon.alpha},
	\begin{equation}    \label{E:m.y1.angle.bound}
              \angle(v_{a}, y_{1}) \leq \frac{2 \epsilon}{n-2k - a}
		\leq \theta_{0} := \frac{2 \epsilon_{0}}{n-2k - a_{0}} 
		  = 2 \alpha < \frac{\pi}{2}.
	\end{equation}
By assumption (\eqref{E:U.epsilon.defn}), $x \in \clU_{I}(\epsilon)$. If $x \in \Pf_{k}$ then we may let $\epsilon \downarrow 0$ in \eqref{E:angle.sum.<.eps} and \eqref{E:m.y1.angle.bound} and get $v_{a} = y_{1} = \cdots = y_{n-k}$. 
Permuting $y_{1}, \ldots, y_{n}$ does not change 
$G_{a}( \cdot ; x)$, we have that 
    \begin{equation}  \label{E:m.a.and.n-k.equal.obs}
        m_{a} \text{ has exactness of fit of order } k .
    \end{equation}
(See definition \ref{D:exactness.of.fit}.) Hence, $v_{a} = y_{1} = \cdots = y_{n-k}$ 
if $x \in \Pf_{k}$. Recall that $\Delta$ is the diagonal map, 
$\Delta : y \in S^{\nvar} \mapsto (y, ..., y) \in \T$.  
Let $\Theta := m_{a} \restriction_{\T}$, so 
    \begin{equation}  \label{E:aug.median.Theta.cont.on.T}
      \Theta \text{ is defined and continuous on } \T .
    \end{equation} 
Since $ma_{a}$ has exactness of fit, we have 
$\Theta \circ \Delta = m_{a} \circ \Delta = $ identity on $S^{\nvar}$. Thus, trivially, 
    \begin{equation}  \label{E:homot.holds.for.ma}
      \eqref{E:Theta.circ.Delta.homot.to.id.on.sphere} 
        \text{ holds for } m_{a} .
    \end{equation}

Now drop the assumption that $x \in \Pf_{k}$ and only demand that 
$x \in \clU_{I}(\epsilon)$. $v_{a}$ is an arbitrary point in $M_{a}(x)$. Therefore, by \eqref{E:m.y1.angle.bound}, 
$M_{a}(x) \subset \overline{V_{y_{1}, \theta_{0}}} \subset V_{y_{1}, \pi/2}$. (See \eqref{E:sphere.cap.defn} for definition.) By \eqref{E:Ma(x).in.Y}, 
$M_{a}(x) \subset Y(x) := \{y_{0}, \ldots, y_{n}\}$. (See \eqref{E:Y(x).defn}.) Let $J \subset N := \{0, \ldots, n\}$ 
satisfy $y_{i} \in M_{a}(x)$ if and only if $i \in J$. Thus, if $j \in J^{c} := N \setminus J$ then 
$G_{a}(y_{j}; x) > G_{a}(v_{a}; x) = G_{a}(y_{i}; x)$ for every $i \in J$. 
Pick $\theta \in (\theta_{0}, \pi/2)$ fixed. Then there exists $\delta > 0$ s.t.\ 
if $x' = (y'_{0}, \ldots, y'_{n}) \in \D$ and $\rho(x', x) < \delta$ (recall $\rho =$ metric 
on $\D$; \eqref{E:rho.is.top.metric}), then
 for every $j \in J^{c}$ we have $G_{a}(y_{j}; x') > \max \{ G_{a}(y_{i}; x') : i \in J \}$ and 
$\angle(y'_{i}, y_{i}) < \theta - \theta_{0}$ ($i \in J$), so, 
$M_{a}(x') \subset \{y'_{i} : i \in J \}$. Let $i \in J$ so we may assume $v_{a} = y_{i}$. Then, by \eqref{E:m.y1.angle.bound},  
$\angle(y'_{i}, y_{1}) \leq \angle(y'_{i}, y_{i}) + \angle(y_{i}, y_{1}) = \angle(y'_{i}, y_{i}) + \angle(v_{a}, y_{1}) < \theta < \pi/2$.
In particular, if $x' \in \D'$ and $\rho(x', x) < \delta$, then
    \begin{equation}  \label{E:ma.in.V.theta}
      m_{a}(x') \in \subset \overline{V_{y_{1}, \theta}} \subset V_{y_{1}, \pi/2}.
    \end{equation}
I.e., if $x' \in \D'$ is in the ball $B_{\delta}(x) \subset \D$ about $x$ 
with radius $\delta$, then $m_{a}(x') \subset \overline{V_{y_{1}, \theta}}$. The preceding argument goes through for $x$ in $\clU_{I}(\epsilon)$ for any 
$I \subset \NN_{n}$ with cardinality $n-k$.

Now suppose $x \in \Pf_{k}$ with $y_{1} = \cdots = y_{n-k}$. Suppose $x' \to x$ through 
$\D'$. By letting $\epsilon_{0} > 0$ be arbitrarily small we make 
$\theta_{0}$ arbitrarily small. But for sufficiently small $\delta$, if $\rho(x', x) < \delta$ then, by \eqref{E:ma.in.V.theta}, $m_{a}(x') \in V_{y_{1}, \theta_{0}}$, i.e., by \eqref{E:sphere.cap.defn}, $\angle \bigl( m_{a}(x'), y_{1} \bigr) < \theta_{0}$. We conclude
    \begin{equation} \label{E:ma.has.no.sings.in.Pk}
       m_{a} \text{ has no singularities in } \Pf_{k} .
    \end{equation}
The claim \eqref{E:ma.has.no.svr.sings.near.Pk} now follows from \eqref{E:V.theta.defn} and definition \ref{D:V-severe}.

Let $\Ss_{a}$ be the singular set of $m_{a}$ (it may not be closed). 
Let $\Ss_{a}^{\msf{V}_{\pi/2}}$ denote the set of $\msf{V}_{\pi/2}$-severe singularities 
of $m_{a}$. Then, $\Phi = m_{a}$ satisfies \eqref{E:no.severe.sings.near.Pk} 
($\Ss_{a}^{\msf{V}_{\pi/2}} \cap \Pf_{k} = \varnothing$). We wish to apply corollary \ref{C:Haus.measure.of.severe.sings.on.spheres} to $m_{a}$. By \eqref{E:homot.holds.for.ma}, we have that
\eqref{E:Theta.circ.Delta.homot.to.id.on.sphere} ($\Theta \circ \Delta$ homotopic to identity) holds for $m_{a}$. By \eqref{L:D'.is.dense.for.aug.median}, $(\Phi, \D')$ with 
$\Phi = m_{a}$ satisfies \eqref{E:D'.=.D.less.S} ($\D' = D \setminus \Ss$). We have already observed that $m_{a}$ satisfies \eqref{E:assume.Phi.S'.sym} (symmetric in arguments).
By \eqref{E:aug.median.Theta.cont.on.T}, $\Theta$  is defined and continuous on $\T$.
By \eqref{E:ma.has.no.sings.in.Pk}, \eqref{E:loc.meas.has.no.severe.sings.near.T} 
($\Ss_{a}^{\msf{V}_{\pi/2}} \cap \T = \varnothing$) also holds. Therefore, by corollary \ref{C:Haus.measure.of.severe.sings.on.spheres}, \eqref{E:Haus.measure.severe.sings.on.spheres} holds:  
Let $R := dist_{n - 2}(\Ss_{a}^{\msf{V}_{\pi/2}}, \T) > 0$. Then for some $\gamma > 0$ not depending on $a$.
    \begin{multline} \label{E:ma.has.exact.fit.etc.}
       \text{If } k \text{ and } a 
         \text{ satisfy } \eqref{E:ma.n.k.a.sizes} 
           \text{ and } \eqref{E:.for.aug.median.a.not.integer}, \\
             \text{ then } m_{a} 
             \text{ has exactness of fit of order } k \text{ and } \\
               \Hm^{n-2}(\Ss_{a}) \geq \Hm^{n-2}(\Ss_{a}^{\msf{V}_{\pi/2}}) 
                 \geq \gamma R^{n-2}. 
                   \text{ In particular, codim} \, \Ss_{a} \leq 2.
    \end{multline}
Here, $\gamma > 0$ is a constant depending only on $n$. In the next section, with $k=1$, we will derive a bound on $\Hm^{n \nvar - \nvar -1}(\Ss_{a})$ similar to the preceding, but computing distance to $\Pf_{k}$ instead of to $\T$.

\section{Distance from severe singularities to $\Pf_{k}$}  \label{SS:dist.from.anglar.median.to.sev.sings.to.Pk} Next, \emph{for } $k=1$, we show more or less the opposite of \eqref{E:ma.has.no.svr.sings.near.Pk}, \emph{viz.}, by allowing $a$ to approach 
$n-2k = n-2$ (see \eqref{E:ma.n.k.a.sizes}), the severe singularities of $m_{a}$ come arbitrarily close to $\Pf_{k}$. Recall the definition, \eqref{E:angle.between.vectors}, 
of $\angle$. If $x=(y_{1}, \ldots, y_{n})$ and $x'=(y_{1}', \ldots, y_{n}')$ are in $\D$, define a distance, $\sigma(x,x')$, from $x$ to $x'$ to be 
	\begin{equation}  \label{E:Delta.metric.for.loc.on.circle}
		\sigma(x,x') := \sum_{i=1}^{n} \angle(y_{i}, y_{i}') .
	\end{equation}
(Another proposal is described in remark \ref{R:another.metric}.)  
By \eqref{E:geod.on.prod.of.spheres} and \eqref{E:rho.is.top.metric}, the topological metric on $\D$ determined by the Riemannian metric on $\D$ is
$\rho(x, x') = \Bigl| \bigl( \angle(y_{1}, y_{1}'), \ldots, \angle(y_{n}, y_{n}') \bigr) \Bigr|$, the Euclidean length of the vector of angles. (See also \eqref{E:geod.dist.on.D.2} in appendix \ref{Chptr:rob.loc.circle.cones.appendix}.) Therefore, by \eqref{E:n.c.sqrd.sum.ineq}, 
    \begin{equation*}
       \rho(x, x') \leq \sigma(x,x')  \leq \sqrt{n} \, \rho(x, x') .
    \end{equation*}
Thus, by definition of Hausdorff measure (appendix \ref{Chptr:Lip.Haus.meas.dim}), Hausdorff measure computed using $\rho$ or $\sigma$ can only differ by at most a constant multiple. Consequently, dimension computed w.r.t.\ the two measures are the same. Hence, our results concerning dimension (chapter \ref{Chptr:spherical.location}) or measure (chapter \ref{Chptr:Haus.meas.of.sing.set}) of singular sets yield the same results no matter which metric we use.

Recall the definition \eqref{E:V.theta.defn} of $\msf{V}_{\theta}$ singularities. 
Using $\sigma$ as our metric, we show the $\Hm^{n-2}$-essential distance between 
$\Pf_{1}$ and the set of 
$\msf{V}_{\pi/2}$-severe singularities of $m_{a}$ can be made arbitrarily small by taking $a$ close to $n-2$ (see \eqref{E:ma.n.k.a.sizes}). This does not contradict \eqref{E:ma.has.no.svr.sings.near.Pk} because \eqref{E:ma.has.no.svr.sings.near.Pk} holds for $a$ bounded away from $n-2k = n-2$. Now we let $a \uparrow n-2$. 

As a first step, we prove the following. See appendix \ref{Chptr:misc.proofs} for the proof.

    \begin{prop}   \label{P:severe.sings.of.ma.dist} 
    If $n \geq 4$ we have
	\begin{equation}  \label{E:dist.to.Pk.O(n-2-a)}
		dist_{n-2} \bigl(\Ss_{a}^{\msf{V}_{\pi/2}}, \Pf_{1} \bigr) = O( n - 2 - a ) 
		              \text{ as } a \uparrow n-2 \text{ (through non-integer values) }.
	\end{equation}
    \end{prop}

For the augmented directional mean $\mu_{t}$ (chapter \ref{Chptr:aug.direct.mean}), the test pattern space $\Pf_{0} = \T$ is the natural perfect fit space. Let $t \in [0, n)$ (see \eqref{E:a.tween.0.n}). For a measure of location with order of exactness of fit 
$\ell \in [0, n/2)$ the space 
$\Pf_{\ell}$ is the natural perfect fit space. Let $\Ss_{\mu_{t}}^{\msf{V}_{\pi/2}}$ denote the singular set of the augmented directional mean $\mu_{t}$. By \eqref{E:Sa=Sa'} every singularity of $\mu_{t}$ is $\msf{V}_{\pi/2}$-severe. 
I.e., $\Ss_{\mu_{t}}^{\msf{V}_{\pi/2}} = \Ss_{\mu_{t}}$, the set of all singularities 
of $\mu_{t}$. As discussed in section \ref{SS:measure.distance.to.P}, for any data map we would like the distance of its set of severe singularities to its perfect fit space to be large while the measure of that set to be small. (But remark \ref{R:fit.instab.tradeoff} points out that these two desiderata cannot be achieved simultaneously.)  Here we examine this issue both for $\mu_{t}$ and general measures, $\Phi$, of location on $S^{1}$ that have positive order of exactness of fit. We find that in sufficiently extreme cases, $\mu_{t}$ dominates $\Phi$ in both respects.

The following suggests that measures of location on $S^{1}$ with positive order of exactness of fit have comparatively large singular sets. However, this may be a price worth paying in order to get resistance to outliers. Resistance might be another thing one might want to hold equal in judging methods in terms of the sizes of their singular sets.  
Recall the definition, \eqref{E:essential.dist.defn}, of essential distance.

Recall lemma \ref{L:solve.rho.t=r}. See appendix \ref{Chptr:misc.proofs} for the proof of the following. 
  \begin{prop}  \label{P:aug.direct.mean.beats.robst.loc}
Let $n > 3$. Take $\nvar =1$ so $d - \nvar -1 = n \nvar-\nvar-1 = n-2$. 
If $R > 0$ let $F_{R, 1}$ denote the collection of maps $\Phi : \D \partlyto S^{1}$ having the following properties. 
	\begin{itemize}
		\item $\Phi$ satisfies the hypotheses 
		  of lemma \ref{L:loose.exactness.of.location.fit} with $k = 1$. 
		  In particular,  
		$\Ss^{\msf{V}_{\pi/2}} \cap \Pf_{1} = \varnothing$. (Just as you would expect, 
		$\Ss^{\msf{V}_{\pi/2}}$ is the set of $\msf{V}_{\pi/2}$-severe singularities 
		of $\Phi$.)
		\item Thus the distance, \emph{a fortiori} the 
		    $\Hm^{n-2}$-essential distance, 
		$dist_{n-2}(\Ss^{\msf{V}_{\pi/2}}, \Pf_{1})$, 
		from $\Ss^{\msf{V}_{\pi/2}}$ to $\Pf_{1}$ is positive. It is strictly less than $R$.
	\end{itemize} 
Thus, $F_{R, 1}$ is increasing in $R$. \emph{THEN:}
  \begin{enumerate}
    \item For any $R > 0$, $F_{R, 1}$ is non-empty.  \label{I:F.R.1.nonempty}
    \item  Let $\delta > 0$ be arbitrary but fixed. Then for $R > 0$ sufficiently small, the following holds. Suppose $\Phi \in F_{R, 1}$ and let 
$r = dist_{n-2} (\Ss^{\msf{V}_{\pi/2}}, \Pf_{1}) (< R)$. There exists 
$t_{r} = n - \tfrac{1}{2} (r/\delta)^{2} + O(r^{4}/\delta^{4}) < n$ s.t., 
if $\Ss_{\mu_{t_{r}}}$ is the set of all singularities of the augmented directional \emph{mean} $\mu_{t_{r}}$, we have 
	\begin{equation}   \label{E:robust.dist.vol.ineqs}
		dist_{n-2} \bigl(\Ss_{\mu_{t_{r}}}, \T \bigr) = \delta^{-1} 
		            dist_{n-2} (\Ss^{\msf{V}_{\pi/2}}, \Pf_{1} )  
		\text{ but } 
		  \Hm^{n-2} (\Ss^{\msf{V}_{\pi/2}} ) . 
		                > \delta^{n-2} \, \Hm^{n-2} \bigl(\Ss_{\mu_{t_{r}}} \bigr). 
	\end{equation}
     \label{I:mu.a.beats.the.shit.out.of.Phi}
   \end{enumerate}
  \end{prop}

By \eqref{E:Sa=Sa'} and \eqref{E:S'a.D'a.defns}, we have 
$\Ss_{\mu_{t_{r}}} = \Ss_{\mu_{t_{r}}}^{\msf{V}_{\pi/2}}$, where 
$\Ss_{\mu_{t_{r}}}^{\msf{V}_{\pi/2}}$ is the set of $\msf{V}_{\pi/2}$-severe singularities of $\mu_{t_{r}}$. Note that $\T = \Pf_{0}$. See proposition \ref{P:rate.of.decrease.of.Hm.St.aug.direct.mean}.

  \begin{remark}
Distances to $\T$ and $\Pf_{1}$ may have different meaning even if their numerical values are the same. The $\delta$ factor allows an exchange rate different from unity. However, distance to $\T$ is never smaller than distance to $\Pf_{1}$. For a given 
$\delta > 0$, any number $R > 0$ small enough that \eqref{E:robust.dist.vol.ineqs} holds might be smaller than any that comes up in practice. However, I conjecture that for 
$\delta = 1$, there are values of $R$ small enough that \eqref{E:robust.dist.vol.ineqs} holds yet are still of practical size. 

The interesting case is $\delta < 1$. In that case, $\Ss_{\mu_{t_{r}}}$ is further from $\T$ than $\Ss^{\msf{V}_{\pi/2}}$ is from $\Pf_{1}$, yet has smaller measure. Thus, with $\delta < 1$, $\mu_{t_{r}}$ is superior to $\Phi$ in the two respects we focus on in this book. (Except in some applications distance 
from $\Pf_{1}$ may be of greater interest than distance from $\T$.)

The proposition suggests that a cost of using a method resistant to outliers is an increase in singularity problems. Sometimes it is cost worth paying.
  \end{remark}

  \begin{remark} \label{R:penalty.for.robustness.in.lin.reg?}
Might it be possible to apply the strategy used in proving proposition \ref{P:aug.direct.mean.beats.robst.loc} to robust linear regression methods (remark \ref{R:robust.lin.reg}) to show that there is a penalty in using a robust linear regression method? 
  \end{remark}
  
\chapter{Linear classification}  \label{Chptr:linear.classification}
We next examine linear classification. That subject deserves a detailed examination, but we only discuss it briefly here. For background information, 
see Johnson and Wichern \cite[pp.\ 494--508]{raJdwW92}, 
Anderson \cite[pp.\ 195--223]{Anderson}, 
Christianini and Shawe-Taylor \cite[Chapter 2]{nCjS-T00.svm}, Vapnik (1998), Hastie \emph{et al} \cite{tHrTjF01.statlearn}, and Agresti \cite[Section 4.2, pp.\ 84--91]{aA90.CatDatAnal}.

\section{Linear classification: General theory}  \label{SS:general.lin.class.thy}
In linear classification the data (more precisely the training data) consist of a point cloud in $\RR^{k}$ each point of which is labeled by either $+1$ or $-1$. These labelled points are ``observations'' or ``examples''. So each observation or example has the form 
$z = (x, z)$, with $x \in \RR^{k}$ and $z = \pm 1$. ($x$ is the ``predictor''.) Let $n$ be the number of observations. Thus, a training data set has the form 
$Y^{n \times \nvar} = (X^{n \times k}, z^{n \times 1})$. We will require
    \begin{equation}  \label{E:n,k.in.lin.class}
      n > 1, \, k > 1 , \, \nvar = k+1 .
    \end{equation}
($k \geq n$, i.e.\  ``wide'' data, is allowed. See remark \ref{R:long.vs.wide}.)

One seeks a $(k-1)$-dimensional affine plane $\pi$ passing through the point cloud. So far this sounds like plane-fitting (chapter \ref{Chptr:sings.in.plane.fit}). But now the purpose of the plane is to separate well the positive points and negative ones. I.e., the goal, not always achievable, is that one side of the plane be enriched in positive points, the other side enriched in negative. Moreover, in addition to $\pi$ one also needs to know which side of $\pi$ is the positive side (equivalently, which is the negative). As we will see presently this cannot be a plane-fitting in the sense of chapter \ref{Chptr:sings.in.plane.fit}.

Cover \cite{tmC65.LinIneqPatRecog}, and references therein, investigated when a linear classifier performs optimally on the training data in the sense of perfectly discriminating the positive and negative examples. See remark \ref{R:constant.classifiers} for a discussion of the opposite situation.

Like linear regression (section \ref{SS:lin.reg.and.LS}, \eqref{E:linear.regression.function}) linear classification involves learning a function, call it $\Gamma$, that can then be used to make predictions. (See remark \ref{R:learning.and.predicting}.) $\Gamma$ depends on parameters $b \in \RR$ (the "bias") and a vector $v^{1 \times k}$ (the "weight vector").  Specifically, given an unlabeled $x \in \RR^{k}$ we classify $x$ to be 
    \begin{equation}  \label{E:how.lin.clssfier.works}
      \Gamma(x) : = \Gamma(x ; Y) := sign (b + x \cdot v) .
    \end{equation} 
(Here, $sign(0)$ is unspecified. See \eqref{E:sign.function}. The dot ``$\cdot$'' indicates Euclidean inner product.) $b$ and $v$ are ``learned'' from the training data. And just as in section \ref{SS:lin.reg.and.LS}, here we are interested in the learning phase. Denote by $LC$ an operation that takes training data as input and (usually) produces a function 
$\Gamma$.
 
As remarked above, linear classification is not a form of plane-fitting (chapter \ref{Chptr:sings.in.plane.fit}). Even the map $X \mapsto v^{\perp} \in G(k-1,k)$ (see \eqref{E:superscript.perp.notation}) is not plane-fitting in the sense of that chapter: If the rows of $X$ lay exactly on a unique $(k-1)$-plane and the plane 
$v^{\perp}$ were were parallel to that plane as required by (\ref{E:plane-fitter.defn}(3)) 
(recall \eqref{E:Delta(Y).defn}) then it would be incapable of discriminating the positive and negative data points. 

  \begin{remark}[Linear classification ``failure'']  \label{R:constant.classifiers}
Let $\Gamma$, $b \in \RR$, and $v^{1 \times k}$ be as in \eqref{E:how.lin.clssfier.works}. If $b \neq 0$ and $x \in \RR^{k}$ satisfies $|x| < |b|/|v|$ then $\Gamma(x)$ finds no information in the predictor vector $x$: $\Gamma$ is constant on the ball $B_{|b|/|v|}(0)$. (See \eqref{E:ball.defn}.) In particular, if $v = 0$ then $\Gamma$ is constant on $\RR^{k}$, therefore not discriminating. Let $Y = (X,z)$ be a training data set used to determine $b$ and $v$. Then $v = 0$ suggests that $X$ has no information about $z$, normally an indication of failure of classifier training. This situation is examined in detail in \cite{spE.ConstClassifiers}.

If $v = 0$, then the pair $(b, v)$ does not define a plane and if at $Y$ a learning algorithm $LC$ has $v = 0$ in its output, we regard $LC$ as undefined at $Y$. Otherwise, it is clear we may assume 
    \begin{equation*}
      |v| = 1 .
    \end{equation*}   
  \end{remark}

Let $m \in \NN$ and
    \begin{equation} \label{E:mcl.Xmk.lin.class.defn}
       \X_{m}^{k} := \text{ space of all } m \times k \text{ real matrices.}
    \end{equation}
So $\X_{m}^{k}$ is homeomorphic to $\RR^{mk}$.

In this chapter we study the singularities of the map $LC$ that specifies the affine function, $\Gamma$. The input to $LC$ consists of training data  
$Y^{n \times \nvar} = (X^{n \times k}, z^{n \times 1})$ (so $\nvar = k+1$), where 
$X \in \X_{n}^{k}$ and the entries in $z$ are all $+1$ or $-1$. 

The set of all such matrices $Y$, tentatively call it $\D$, is disconnected, a departure from \eqref{E:D.metric.F.normal}. So we proceed as follows. Let $Y$ be as in the last paragraph. Let $P$ denote the number of $+1$'s in $z$. Let $N$ denote the number of $-1$'s. Hold $P$ (and therefore $N = n-P$) fixed. So we ``condition'' on $P$. If $P=n$ or $N=n$, $Y$ is clearly useless as a training data set. So assume $0 < P < n$. One expects that the typical classification method $LC$ will be invariant under permutation of rows of $Y$, but in any case take the first $P$ entries in $z$ to be $+1$ and the last $N$ entries to be $-1$. Thus, $z$ is fixed. Let $z_{i}$ be the $i^{th}$ entry of $z$ 
($i\in \NN_{n}$; see \eqref{E:N.sub.n}). Let $\D^{z}$ be the set of training data sets $Y$ whose last column is this special $z$: 
    \begin{equation}  \label{E:zP.in.lin.class}
      z := z^{P} = \bigl( z_{1}, \ldots, z_{P}, z_{P+1}, \ldots, z_{n} \bigr) 
        := (+1, \ldots, +1, -1, \ldots, -1)^{T}
    \end{equation} 
with $1_{n} \cdot z = P-N$, where ``$1^{n}$'' is defined in \eqref{E:1n.col.vec.defn}. Now drop the subscript: 
$\D := \D^{z^{P}}$. Thus, 
    \begin{equation}  \label{E:prelim.D.defn.in.lin.class}
      \D := \X_{m}^{k} \times \{ z^{P} \} .
    \end{equation}
Later we will replace this $\D$ by one that is a compact manifold.

Our analysis will not capture sensitivity of $LC$ to changes 
in $z = (\pm 1, \pm1, \ldots, \pm 1)$. The methods like those described in remark \ref{R:discreteness} might be employed to study that and it might be interesting to do so, but we do not attempt that here. 
  
As usual, let $S^{k-1} \subset \RR^{k}$ be the $(k-1)$-dimension unit sphere centered at 0. Define
    \begin{equation}  \label{E:LC.Phi}
      \Phi(Y) := v \in S^{k-1}, \qquad Y \in \D 
    \end{equation}  
whenever possible. Thus, the codomain, $\F$, of $\Phi$ is just $S^{k-1}$.    

Next we define $\Pf$ and $\T$. Let
    \begin{equation*}
      s > 0 .
    \end{equation*} 
Choose arbitrary matrices 
$T_{+}^{P \times k}$ and $T_{-}^{N \times k}$. Later we will require that  
the rows of these matrices have distance less than $s/2$ from some fixed vector 
$\tilde{x}$. For example, we might require their rows to all be 0. Let 
    \begin{equation}  \label{E:lin.class.T.matrix}
      T^{n \times k} := (T_{+}^{T}, T_{-}^{T})^{T} .
    \end{equation}
Thus, the first $P$ rows of $T$ constitute $T_{+}$, the last $N$ rows constitute 
$T_{-}$. 
    \begin{multline*}
      \text{Denote the $i^{th}$ row of $T$ by } x_{+,i} \; (i=1, \ldots, P) 
        \text{ and } x_{-,j} \; (j=P+1, \ldots, n) \\
          \text{ or, generically, by } x_{\pm,i} \;  (i \in \NN_{n}) . 
    \end{multline*}
For $u^{1 \times k} \in S^{k-1}$, 
    \begin{multline}  \label{E:lin.class.X(u)}
      \text{let } X(u)^{n \times k} \text{ be a matrix whose first $P$ rows are }
        T_{+} + s 1^{P} \, u \\
          \text{ and whose last $N = n-P$ rows are } T_{-} - s 1^{N} \, u .
    \end{multline} 
Let 
    \begin{equation}  \label{E:T.and.Pf.in.lin.class}
      \T = \T_{T} := \Pf_{T} 
        := \bigl\{ (X(u),z^{P}) : u \in S^{k-1} \bigr\} .
    \end{equation} 
(Again, $z^{P} = (+1, \ldots, +1, -1, \ldots, -1)^{T}$ is fixed. Do not confuse $\T$ 
with the $n \times k$ matrix $T$.) 

Since $s > 0$, $T_{+}$, and $T_{-}$ are arbitrary, $\Pf$ is not canonical. This is a quality it shares with the $\T$ discussed in section \ref{SS:calibration} and example \ref{Ex:disconnected.F}. Except, possibly, if the classifier $LC$ is regularized (remark \ref{R:regularization.generalities}), it is reasonable to suppose that for some choice 
of $T_{+}$, $T_{-}$, and $s > 0$, if $LC$ is trained on a data set 
$(X,z^{P}) \in \Pf_{T}$, it will correctly classify every row 
in $(X,z^{P})$, the training set. 

We show 
    \begin{equation}  \label{E:lin.class.Pf.homeomto.sphere} 
      \T = \T_{T} \homeomto S^{k-1} .
    \end{equation} 
(``$\homeomto$'' means ``homeomorphic to''): 
Given $Y \in \T_{T}$ one can determine $u \in S^{k-1}$ and \emph{vice versa}. Specifically, let $Y = (X, z^{P})^{n \times \nvar} \in \T_{T}$. 
Then for some $u \in S^{k-1}$, $X= X(u)$. Therefore the rows of $X^{n \times k}$ are 
    \begin{equation}  \label{E:xi.in.lin.class}
      x_{i} := x_{\pm,i} + z_{i} s u \qquad (i \in \NN_{n}) .
    \end{equation} 
Let
    \begin{equation}  \label{E:lin.class.Y(u).defn}
      Y = Y(u) := \bigl( X(u), z^{P} \bigr) 
    \end{equation}
and define
    \begin{equation}  \label{E:Lambda.in.lin.class}
      \Lambda \bigl[ Y(u) \bigr] = s^{-1}(x_{1} - x_{+, 1}) = u .
    \end{equation}
It follows that $Y(\cdot) : S^{k-1} \to \D$ is an imbedding. Therefore, by Boothby \cite[Theorem (5.5), p.\ 78]{wmB75}, $\T$ is an imbedded submanifold of $\D$ and 
    \begin{equation}  \label{E:Lambda.is.lin.class.diffeom}
      \Lambda : \T \to S^{k-1} \text{ is a diffeomorphism. }
    \end{equation}
    
Let $\Phi$ be as in \eqref{E:LC.Phi} and suppose it is defined and continuous on a dense subset $\D'$ of $\D$ satisfying \eqref{E:D'.=.D.less.S}. \emph{Assume:}
    \begin{equation}  \label{E:LinClass.Theta.cont.on.T}
      \text{The restriction } \Phi \restriction_{T} \text{ of $\Phi$ to } 
        \T \text{ has a unique continuous extension $\Theta$ to all of } \T . 
    \end{equation}

Recall \eqref{E:how.lin.clssfier.works}. 
We wish $T_{+} \in \X_{P}^{k}, \; T_{-} \in \X_{N}^{k} \subset \RR^{k}$, 
$s > 0$, and $b \in \RR$ to be s.t.\ the classifier 
$\Gamma(x;Y) := sign \bigl( b + x \cdot \Phi(Y) \bigr)$ ($x \in \RR^{k}$) classifies every row of $Y(u)$ correctly for every $u \in S^{k-1}$. When can we be assured of finding a $b \in \RR$ that makes this possible? 
Let 
    \begin{equation*}
      v = v(u) := \Phi \bigl( Y(u) \bigr) \in S^{k-1} .
    \end{equation*} 
For $u \in S^{k-1}$ and each $i = 1, \ldots, P$ we need to have $b + x_{i} \cdot v(u) > 0$ and for each 
$i = P+1, \ldots, n$ we need $b + x_{i} \cdot v(u) < 0$, where $b = b(Y)$ and $v = \Phi(Y)$ depend on the training data, $Y \in \T$. 

Let $u \in S^{k-1}$. By \eqref{E:xi.in.lin.class}, $x_{i}^{1 \times k} := x_{\pm,i} + z_{i} s u$ is the $i^{th}$ row of $X(u)$. So we want
    \begin{align}  \label{E:classifies.test.patterns.correctly}
      0 < b(u) + v \cdot (x_{\pm,i} + s u) &= 
      b(u) + (v \cdot x_{\pm,i} + s v \cdot u)  \text{ for $i = 1, \ldots, P$ and} \\
      0 > b(u) + v \cdot (x_{\pm,j} - s u) &=
      b(u) + (v \cdot x_{\pm,j} - s v \cdot u)  \text{ for } j = P+1, \ldots, n \notag .
    \end{align}

Let $\tilde{x}^{1 \times k}$ be fixed. Recall \eqref{E:lin.class.T.matrix}. Two possible choices of $\tilde{x}$, but not the only ones, are 0 and the vector of column-wise arithmetic means of $T$: 
$n^{-1} 1_{n} T = n^{-1} \sum_{i=1}^{n} x_{\pm,i}$. Define
    \begin{equation}  \label{E:lin.class.delta.defn}
      \delta := \max \bigl\{ |x_{\pm,i} - \tilde{x}| : i \in \NN_{n} \bigr\} .
    \end{equation} 
Recall \eqref{E:lin.class.Y(u).defn}. Write $v(u) := \Phi \bigl( Y(u) \bigr)$. Suppose 
    \begin{equation} \label{E:Theta(Y).dot.Lambda(Y).big}
      \Theta \bigl( Y(u) \bigr) \cdot \Lambda \bigl( Y(u) \bigr)  
        = v(u) \cdot u > 2 \delta/s > 0 
          \text{ for every } u \in S^{k-1} . 
    \end{equation}
Since $u, v(u) \in S^{k-1}$, this requires
        \begin{equation}  \label{E:s>2.delta}
          s > 2 \delta . 
        \end{equation} 
\eqref{E:Theta(Y).dot.Lambda(Y).big} means $\Theta \bigl( Y(u) \bigr)$ and $\Lambda \bigl( Y(u) \bigr)$ are in rough agreement: 
    \begin{equation*}
      \bigl| \Theta \bigl( Y(u) \bigr) - \Lambda \bigl( Y(u) \bigr) \bigr|^{2} 
        < 2(1-2\delta/s) .
    \end{equation*} 
If $\Phi \bigl( Y(u) \bigr) \cdot \Lambda \bigl( Y(u) \bigr) > 0$ for every $Y \in \T_{T}$, then, by \eqref{E:Lambda.is.lin.class.diffeom}, \eqref{E:LinClass.Theta.cont.on.T}, and compactness of $S^{k-1}$, \eqref{E:Theta(Y).dot.Lambda(Y).big} holds for some $s > 0$.

Recall that $u$ and $v(u)$ are unit vectors and $Y(u)$ is the training data 
$\bigl( X(u), z^{P} \bigr)$. Let $b : S^{k-1} \to \RR$ satisfy 
    \begin{equation}  \label{E:b(u).interval}
      b(u) \in \bigl( -\delta - v(u) \cdot \tilde{x}, \; \delta - v(u) \cdot \tilde{x} \bigr) 
        \text{ for every } u \in S^{k-1} .
    \end{equation} 
Then, by \eqref{E:xi.in.lin.class}, \eqref{E:lin.class.delta.defn}, and \eqref{E:Theta(Y).dot.Lambda(Y).big}, the following holds for every 
$i =1, \ldots, P$ and $j = P+1, \ldots, n$: 
    \begin{multline*}
      v(u) \cdot x_{j} = 
        v(u) \cdot (x_{\pm,j} -  \tilde{x}) + v(u) \cdot \tilde{x} 
          - s v(u) \cdot u \leq \delta + v(u) \cdot \tilde{x} - 2 \delta
          = -\delta + v(u) \cdot \tilde{x} \\
            < -b(u) < \\ 
              \delta + v(u) \cdot \tilde{x} 
               = - \delta + v(u) \cdot \tilde{x} + 2 \delta 
                < v(u) \cdot (x_{\pm,i} -  \tilde{x}) + v(u) \cdot \tilde{x} 
                 + s v(u) \cdot u = v(u) \cdot x_{i} . 
    \end{multline*}
Thus, if \eqref{E:b(u).interval} and \eqref{E:Theta(Y).dot.Lambda(Y).big} hold then so does \eqref{E:classifies.test.patterns.correctly}. 
    
To rephrase, let $\Gamma$ be as in \eqref{E:how.lin.clssfier.works}, where $b$ and $v$ are determined by the training sample $Y$, if defined there. Write $\Phi(Y) = v$ as in \eqref{E:LC.Phi}. To acknowledge the fact that $\Gamma$ is trained on $Y$, 
write $\Gamma_{Y} = \Gamma$. For $u \in S^{k-1}$, for $\Gamma_{Y(u)}$ 
write $b = b(u)$ and $v = v(u)$. Suppose \eqref{E:LinClass.Theta.cont.on.T} holds. Then, defining $\delta$ by \eqref{E:lin.class.delta.defn} and recalling \eqref{E:b(u).interval}, we have,
    \begin{multline}  \label{E:big.Phi.dot.Lambda.big.Gamma.correct.on.T}
      \text{If for every $Y \in \T$ we have } \\
        \Phi(Y) \cdot \Lambda(Y) > 2 \delta/s \text{ and } 
      \bigl| b \circ \Lambda(Y) + v \circ \Lambda(Y) \cdot \tilde{x} \bigr| < \delta ,
      \text{ then \eqref{E:classifies.test.patterns.correctly} holds, } \\
      \text{ i.e. $\Gamma_{Y}$ correctly classifies every observation in } Y \in \T . 
    \end{multline} 

Recall \eqref{E:LinClass.Theta.cont.on.T}. We will assume
    \begin{equation}  \label{E:Theta.dot.Lambda.>.-1}
      \text{If } \Theta(Y) \cdot \Lambda(Y) > -1 \text{ for every } Y \in \T.
    \end{equation} 
\eqref{E:Theta.dot.Lambda.>.-1} is reminiscent of \eqref{E:location.Phi.on.diagnl.loose}. In particular, \eqref{E:Theta.dot.Lambda.>.-1} olds when \eqref{E:Theta(Y).dot.Lambda(Y).big} does. \eqref{E:Theta.dot.Lambda.>.-1} is a very weak form of calibration (section \ref{SS:calibration}).

\emph{Claim:}  
    \begin{equation}  \label{E:lin.class.Theta.star.nontrivial}
      \text{\eqref{E:Theta.dot.Lambda.>.-1} implies }
        \Theta \text{ satisfies \eqref{E:nontriv.r-dim.homol} with } r := k-1 .
    \end{equation} 
Let $Y \in \T$. Regard 
$\Theta(Y), \Lambda(Y) \in S^{k-1}$ as points in $\RR^{k}$. If \eqref{E:Theta.dot.Lambda.>.-1} holds then, for $t \in [0,1]$, we have
    \begin{align*}
      \bigl| (1-t) \Theta(Y) + t \Lambda(Y) \bigr|^{2}
        &= (1-t)^{2}  \bigl| \Theta(Y) \bigr|^{2} + 2 t(1-t) \Theta(Y) \cdot \Lambda(Y)
          + t^{2} \bigl| \Lambda(Y) \bigr|^{2} \\
        &> (1-t)^{2} - 2 t(1-t) + t^{2}  \\
        &= (1-2t)^2 \geq 0 .
    \end{align*}
Thus, we may define a homotopy from $\Theta$ to $\Lambda$ by
    \begin{equation*}
      H(Y,t) := \bigl| (1-t) \Theta(Y) + t \Lambda(Y) \bigr|^{-1}
        \bigl[ (1-t) \Theta(Y) + t \Lambda(Y) \bigr] \in S^{k-1} , 
          \qquad Y \in \T, \, t \in [0,1] .
    \end{equation*}
(This is reminiscent of remark \ref{R:loose.exactness.of.fit.and.homotopy}.) $H$ is continuous by \eqref{E:LinClass.Theta.cont.on.T} and \eqref{E:Lambda.is.lin.class.diffeom}. This proves that $\Theta$ and $\Lambda$ are homotopic. By \eqref{E:Lambda.is.lin.class.diffeom} again, it follows that $\Theta_{\ast} : \tilde{H}_{k-1}(\T) \to \tilde{H}_{k-1}(S^{k-1}) = H_{k-1}(\F)$ is nontrivial. By \eqref{E:n,k.in.lin.class}, $k > 1$ so $\tilde{H}_{k-1} = H_{k-1}$. Therefore, the claim \eqref{E:lin.class.Theta.star.nontrivial} holds.

So far $\D \homeomto \RR^{nk}$ is not a compact manifold. 
We replace $\D$ by analogues of the spaces $\D_{\infty}$ and $\D_{\mu}$ in section \ref{SS:D.T.plane.fit}. Recall \eqref{E:mcl.Xmk.lin.class.defn}. 
Write $\D = \X_{n}^{k} \times \{ z^{P} \}$. Let 
    \begin{equation}  \label{E:LDA.X.infty.defn}
      \X_{\infty} =  \text{ the one-point compactification of } \X_{n}^{k} .. 
    \end{equation}
Thus, $\X_{\infty}$ is diffeomorphic to $S^{nk}$. In fact, as spelled out in section \ref{SS:D.T.plane.fit} for $\D_{\infty}$, the space $\X_{\infty}$ \emph{is} a round sphere. Define 
    \begin{equation}  \label{E:LDA.D.infty.defn}
      \D_{\infty} := \X_{\infty} \times \{z^{P}\} .
    \end{equation}
So 
    \begin{equation}  \label{E:lin.class.D.infty.is.sphere}
      \D_{\infty} \text{ is diffeomorphic to } S^{nk} .
    \end{equation}

Alternatively, let $\mu : S^{n k-1} \to (0, \infty)$ be smooth (e.g., constant). Define 
	\begin{equation}  \label{E:LDA.X.mu.defn}
	      \X_{\mu} := \Bigl\{ X \in \X_{n}^{k} \setminus \{0\} :  
	        \| X \| = \mu \bigl( \| X \|^{-1} X \bigr) \Bigr\} ,
	\end{equation}
where $\| \cdot \|$ is the Euclidean norm defined in \eqref{E:matrix.norm}. This is analogous to \eqref{E:plane.fitting.D.mu.defn}. Assume for $\X_{\mu}$ the analogue of \eqref{E:metrics.on.D.mu}. 
Let 
    \begin{equation}  \label{E:LDA.D.mu.defn}
      \D_{\mu} :=  \X_{\mu} \times \{ z^{P} \} .
    \end{equation} 
By \eqref{E:D.mu.diffeo.to.sphere}, 
    \begin{equation}  \label{E:LinClass.D.mu.diffeom.to.sphere}
      \D_{\mu} \text{ is diffeomorphic to an $(nk-1)$-dimensional sphere.}
    \end{equation} 

Actually, as in remark \ref{R:shift.D.mu}, we can shift $\D_{\mu}$. This is accomplished by shifting $\X_{\mu}$. Let $\hat{x}$ be a $k$-row vector. 
Replace $\X_{\mu}$ by $\X_{\mu} + 1^{n} \hat{x}$. In particular $T$ (see \eqref{E:lin.class.T.matrix}) gets replaced by $T + 1^{n} \hat{x}$ and, 
for $u \in S^{k-1}$, $X(u)$ (see \eqref{E:lin.class.X(u)}) gets replaced 
by $X(u) + 1^{n} \hat{x}$. Our analysis above is undisturbed if we replace 
$\tilde{x}$ (in \eqref{E:lin.class.delta.defn}, for example) by $\tilde{x} + \hat{x}$. 
 $\D_{\mu}$, perhaps after shifting, can be used to ``localize'' the singular set as in remarks \ref{R:restricting.plane.fitter.to.sphere} and \ref{R:dnsity.contrs.as.data.spaces}. 

    \begin{equation}  \label{E:D.X.options.in.lin.class}
      \text{ Let $\X$ be either $\X_{\infty}$ or } \X_{\mu} .
      \text{ Let $\D = \X \times \{ z^{P} \}$ be either $\D_{\infty}$ or } \D_{\mu} .  
    \end{equation}
Metrize $\D$ by the metric on $\X$. Whether $\D = \D_{\infty}$ 
or $\D = \D_{\mu}$, interpret $\Phi$, $\D'$, 
$\Ss := \D \setminus \D'$, and $\T$ in the obvious manner as in section \ref{SS:D.T.plane.fit}. 
(If $\D = \D_{\infty}$, the point at infinity is included in $\Ss$.) 
Remark \ref{R:different.measures.on.D} immediately translates to the linear classification setting.

 \begin{remark}  \label{R:Pf.in.lin.class}
Recall \eqref{E:D.X.options.in.lin.class}. Let $X \in \X$. This time let $x_{i}$ be the $i^{th}$ row of $X$ ($i \in \NN_{n}$). Let $Y = (X, z^{P}) \in \D$ and suppose $\Gamma$ defined in \eqref{E:how.lin.clssfier.works} classifies every row of $X$ correctly. 
So for $1 \leq i \leq P < j \leq n$ we have $b + v \cdot x_{i} > 0 > b + v \cdot x_{j}$. 
Hence, $0 < (b + v \cdot x_{i}) - (b + v \cdot x_{j}) = (x_{i} - x_{j}) \cdot v$. 

For $\epsilon \in (0,1)$ define $\Pf_{\epsilon}$ to be the space of $Y = (X,z^{P})$ s.t.\ there exists $v = v[Y] \in S^{k-1}$ (depending on $Y$) having the following property. 
If $1 \leq i \leq P < j \leq n$ then 
    \begin{equation}  \label{E:lin.class.Pf.eps.ineq}
      (x_{i} - x_{j}) \cdot v[Y] > (1-\epsilon) |x_{i} - x_{j}| . 
    \end{equation} 
Then $\Pf_{\epsilon}$ is open.  

Let $Y = (X,z^{P}) \in \Pf_{\epsilon}$ and 
$c := (1-\epsilon) \min_{i \leq P < j} |x_{i} - x_{j}| > 0$. Because the inequality in \eqref{E:lin.class.Pf.eps.ineq} is strict, we must have $c > 0$. 
Let $i,j \in \NN_{n}$ satisfy $i \leq P < j$. Let $v = v[Y]$. Then 
$x_{i} \cdot v > c + x_{j} \cdot v$. Thus, for small $\epsilon$, 
$x_{i} \cdot v > c + x_{j} \cdot v$ are well separated for all $i \leq P < j$, relative to the lengths of the difference, $x_{i} - x_{j}$. Let $m := \max_{j > P} \, (x_{j} \cdot v)$. 
Then $x_{i} \cdot v > c + m \geq c + x_{j} \cdot v > x_{j} \cdot v$ 
($1 \leq i \leq P < j \leq n$). Hence, if $b = - c-m$, then 
$b + (x_{i} \cdot v) > 0 > b + x_{j} \cdot v$. Thus, the classifier $\Gamma$ in \eqref{E:how.lin.clssfier.works} with these choices of $b$ and $v$ classifies, i.e. fits, $Y$ perfectly. But the smaller $\epsilon$ is the ``more perfect'' it is. (This concept is discussed in remark \ref{R:complexity}.) 

Let $Y \in \T$. Recall \eqref{E:T.and.Pf.in.lin.class} and \eqref{E:lin.class.Y(u).defn}. There exists $u \in S^{k-1}$ s.t.\ $Y = Y(u)$. Then $x_{i}$ is given by \eqref{E:xi.in.lin.class} ($i \in \NN_{n}$). Thus, if $1 \leq i \leq P < j \leq n$, we have $x_{i} - x_{j} = (x_{\pm,i} - x_{\pm,j}) + 2 s u$. Recall \eqref{E:Lambda.in.lin.class} and take $v[Y] :=  \Lambda(Y) = u$. Define $\delta$ as in \eqref{E:lin.class.delta.defn} and assume \eqref{E:s>2.delta}. Then no matter what $\tilde{x}$ is, we have $s > 2 \delta \geq |x_{\pm,i} - x_{\pm,j}|$  so $|x_{i} - x_{j}| \leq |x_{\pm,i} - x_{\pm,j}| + 2s \leq 2 \delta + 2 s < 3 s$. Therefore, 
    \begin{equation*}
      (x_{i} - x_{j}) \cdot v[Y] = (x_{\pm,i} - x_{\pm,j}) \cdot u + 2 s 
        \geq 2 s - 2 \delta 
          = s + (s - 2 \delta) > s 
            \geq \frac{s}{2 s + 2 \delta} |x_{i} - x_{j}| 
    \end{equation*}
Since $2 \delta \in [0, s)$, we see that \eqref{E:s>2.delta} implies that $\epsilon$ in \eqref{E:lin.class.Pf.eps.ineq} can range from $1/2$ to $2/3$. 
  \end{remark}

$\T$ has a neighborhood in $\D$ with a retraction onto $\T$ constructed as follows. 
Let $\clU$ be an open neighborhood, of $\T$ in $\D$ sufficiently small that, even 
if $\D = \D_{\infty}$, all elements of $\clU$ are finite. We may extend $\Lambda$ defined in \eqref{E:Lambda.in.lin.class} to all of $\clU$ as follows. Let $Y = (X, z^{P}) \in \clU$ with $x_{1}$ the first row of $X$. Define 
$\Lambda(Y) = s^{-1}(x_{1} - x_{+, 1}) \in \RR^{k}$. Making $\clU$ smaller if necessary, we may assume that for every $Y \in \clU$ we at least have 
$\Lambda(Y) \neq 0$. This defines a continuous function on $\clU$, but 
$\Lambda(\clU)$ does not lie in $S^{k-1}$.  

Recall \eqref{E:zP.in.lin.class}. For $Y \in \clU$, define $f(Y) \in \T$ to be the $n \times \nvar$ matrix $(X,z^{P})$ whose last column is $z^{P}$ and with the $i^{th}$ row 
of $X^{n \times k}$ given by 
$x_{\pm,i} + z_{i} \, s \bigl| \Lambda(Y) \bigr|^{-1} \Lambda(Y)$ ($i \in \NN_{n}$). Recall \eqref{E:lin.class.X(u)} and \eqref{E:T.and.Pf.in.lin.class}. 
Thus, $f(Y) \in \T$ for every $Y \in \clU$ and, if $Y \in \T$, we have 
$f(Y) = Y$. I.e., $f : \clU \to \T$ is a retraction. (See remark \ref{R:retraction.in.manifs}.)

As mentioned, typically, the training operation, $LC$, will be invariant under permutation of the examples in the training data. We do not need to assume that here. Instead, we take the group, $G$, in chapters \ref{Chptr:Haus.meas.of.sing.set} and \ref{Chptr:severity} to be trivial.  

Let $\Ss$ be the singular set of $\Phi$. 
    \begin{equation}  \label{E:Lin.Class.S.has.emp.int}
      \text{We assume $\Ss$ has empty interior.}
    \end{equation} 
We apply the ``severity trick'' (remark \ref{R:severity.trick}). Since $\F$ is a sphere, we can define covers and convex combination function as in section \ref{S:convex.combos.on.spheres}. Define $\msf{V}_{\pi/2}$ by \eqref{E:V.theta.defn} 
with $\theta = \pi/2$. Use the convex combination function on $\msf{V}_{\pi/2}$ given by \eqref{E:lin.combo.on.sphere}. As specified by \eqref{E:S^V.notation} we use 
$\Ss^{\msf{V}_{\pi/2}} \subset \Ss$ to denote the set 
of $\msf{V}_{\pi/2}$-severe singularities. By \eqref{E:SV.is.closed}, 
$\Ss^{\msf{V}_{\pi/2}}$ is closed.

With $\D = \D_{\mu}$ or $\D = \D_{\infty}$ and $\D' := \D \setminus \Ss$ assume: $\D'$ is dense in $\D$, $\D' \cap \Pf$ is dense in $\Pf = \T_{T}$. Thus, 
$\Ss^{\msf{V}_{\pi/2}}$ has empty interior. Asssume
\eqref{E:LinClass.Theta.cont.on.T} holds and  
    \begin{equation}  \label{E:Lin.Class.SV.pi/2.disjoint.from.T}
        \Ss^{\msf{V}_{\pi/2}} \cap \T = \varnothing . 
    \end{equation}
Recall that $\T$ is a neighborhood retract. Then $(\Phi, \Ss)$, satisfies the hypotheses of theorem \ref{T:if.lin.combo.on.F.then.can.rstrct.to.bad.sings},  
with $G = \{ \text{identity on } \D\}$, $\msf{V} = \msf{V}_{\pi/2}$, and $\Pf = \T$. Therefore, there exists 
$\Omega : D \partlyto \F$ s.t.\ $\Omega \restriction_{\T}$ is continuous and equals 
$\Theta$. Therefore, $\Omega$ satisfies \textbf{hypothesis \ref{Hyp:extend}} of theorem \ref{T:Phi.star.Hr.contains.Theta.star.Hr}. 
The singular set of $\Omega$ 
(relative to $\tilde{\D} := \D \setminus \Ss^{\msf{V}_{\pi/2}}$) is a subset 
of $\Ss^{\msf{V}_{\pi/2}}$.   

$\Ss^{\msf{V}_{\pi/2}} \cap \T = \varnothing$ is much more plausible than 
$\Ss \cap \T = \varnothing$. By compactness, \eqref{E:Lin.Class.SV.pi/2.disjoint.from.T} and \eqref{E:dist.s,dist.0} imply 
$dist_{d-k}(\Ss^{\msf{V}_{\pi/2}}, \T) \geq dist(\Ss^{\msf{V}_{\pi/2}}, \T) > 0$. 
By \eqref{E:Lin.Class.SV.pi/2.disjoint.from.T}, $\Ss^{\msf{V}_{\pi/2}}$ satisfies \textbf{hypothesis \ref{Hyp:S.cap.T.small}} of theorem \ref{T:Phi.star.Hr.contains.Theta.star.Hr}. 
$(\Omega, \Ss^{\msf{V}_{\pi/2}})$ obviously satisfies \textbf{hypothesis \ref{Hyp:T.manif}}. By \eqref{E:SV.is.closed} and \eqref{E:Lin.Class.S.has.emp.int}, 
\textbf{hypothesis \ref{Hyp:S.cap.T.closed}} of theorem \ref{T:Phi.star.Hr.contains.Theta.star.Hr} holds for
$(\Omega, \Ss^{\msf{V}_{\pi/2}})$.

Let 
    \begin{equation*}
      r = t = k-1 
    \end{equation*}
(see \eqref{E:Lambda.is.lin.class.diffeom}) so \textbf{hypothesis \ref{Hyp:r.integer}} of theorem \ref{T:Phi.star.Hr.contains.Theta.star.Hr}) is satisfied. Assume \eqref{E:Theta.dot.Lambda.>.-1}. Then, by \eqref{E:lin.class.Theta.star.nontrivial}, 
$\Theta$ satisfies \eqref{E:nontriv.r-dim.homol} with $r = k-1 > 0$ (by \eqref{E:n,k.in.lin.class}). 

Now, $\D \homeomto S^{nk}$ or $S^{nk-1}$. 
And with $d = nk$ or $d = nk-1$, by \eqref{E:n,k.in.lin.class} and Munkres \cite[Corollary 47.2, p.\ 281]{jrM84}, we have 
$\check{H}^{d-r}(\D) \isomto H^{d-r}(\D) = \{ 0 \}$. 
Therefore, proposition \ref{P:sing.dim.when.H.d-r.D.=.0} applies, with $\Omega$ in place of $\Phi$ and \eqref{E:codim.S.leq.r+1} holds with $\Ss^{\msf{V}_{\pi/2}}$ in place of $\Ss'$. I.e., we have proved.
         \begin{multline}  \label{E:sing.codim.in.lin.class}
            \Hm^{d-k}(\Ss) \geq \Hm^{d-k}(\Ss^{\msf{V}_{\pi/2}}) > 0
              \; (d = nk \text{ or } d = nk-1) .  \\
               \; \text{ In particular, } 
                 \text{codim} \, \Ss \leq \text{codim} \, \Ss^{\msf{V}_{\pi/2}} \leq k .
         \end{multline}

But \eqref{E:sing.codim.in.lin.class} holds for any data map 
$\Psi : \D \partlyto \F = S^{k-1}$ whose singular set is closed, has empty interior, is disjoint from $\T$ and whose restriction to $\T$ equals $\Theta$. (See example \ref{Ex:basic.property.case}.) Thus, $\Omega$ has property \ref{Pty:agree.near.T} 
with $G = \{ \text{identity on } \D\}$ and $a = d - r -1 = d - (k-1) - 1 = d-k$. By \eqref{E:lin.class.D.infty.is.sphere} and \eqref{E:LinClass.D.mu.diffeom.to.sphere}, we can apply lemma \ref{L:biLip.triangulation} to conclude that it has a bi-Lipschitz triangulation. Indeed, it is triangulable by the $d$-skeleton (the boundary) of a $(d+1)$-simplex (as in section \ref{SS:D.T.plane.fit}). 

By \eqref{E:Lambda.is.lin.class.diffeom}, $\T$ is diffeomorphic to a sphere. Therefore, by example \ref{Ex:TubeNbhdSpecialCaseOfCones}, $\Pf$ has a neighborhood in the restriction $T \D \restriction_{\Pf} = T \D \restriction_{\T}$ that is fibered by cones as in definition \ref{D:fibering.by.cones}.

Thus, $\Phi$ satisfies the hypotheses of theorem \ref{T:lwr.bnd.on.Haus.meas} 
with $\Pf = \T$. Note that, by \eqref{E:n,k.in.lin.class} and with $p = t = k-1$, we have $d-p-1 = d-k = a$.) Therefore, for some $\gamma > 0$ not depending on $\Phi$ we have, 
  \begin{prop}  \label{P:measure.bnd.for.lin.class}
Suppose \eqref{E:Theta.dot.Lambda.>.-1} and \eqref{E:Lin.Class.SV.pi/2.disjoint.from.T} hold. Then there exists $\gamma > 0$ depending only 
on $\T = \T_{T}$ and $\D$ s.t.\ $\text{codim} \, \Ss^{\msf{V}_{\pi/2}} \leq k$, 
$dist_{d-k}(\Ss^{\msf{V}_{\pi/2}}, \T) > 0$, and 
   \begin{equation}  \label{E:measure.ineq.for.lin.class}
      \Hm^{d-k}(\Ss) \geq \Hm^{d-k}(\Ss^{\msf{V}_{\pi/2}}) 
        \geq \gamma R^{d-k} > 0 , 
          \text{ where } 0 < R \leq dist_{d-k}(\Ss^{\msf{V}_{\pi/2}}, \T) .
    \end{equation}
  \end{prop}
Here, $\Hm^{d-k}$ is defined based on the metric on 
$\D = \D_{\infty}$ or  $\D_{\mu}$ induced by its Riemannian metric. 

Suppose a data scientist wants to train a classifier based on a data set $Y$. Since this is a chapter about linear classification, let us assume that the data scientist (or AI) has decided in advance, i.e. without looking at $Y$, to fit a linear classifier of some sort satisfying \eqref{E:Theta(Y).dot.Lambda(Y).big} and \eqref{E:Lin.Class.SV.pi/2.disjoint.from.T}, but may not have settled on a specific one. (To make this decision without looking at $Y$ is dangerous. But let us ignore that difficulty.) But he, she, or it seeks a linear classifier satisfying \eqref{E:Theta(Y).dot.Lambda(Y).big} and \eqref{E:Lin.Class.SV.pi/2.disjoint.from.T}.  

Suppose, having made that decision, the data scientist now examines the data before choosing a specific classifier. (The next section analyzes a specific example.) It doesn't matter how that choice is made. The map 
$\Phi$ captures \emph{the entire process of choosing a linear classifier and applying it.} This process might include non-algorithmic elements but still satisfies \eqref{E:Theta(Y).dot.Lambda(Y).big} and \eqref{E:Lin.Class.SV.pi/2.disjoint.from.T}.  
In any case proposition \ref{P:measure.bnd.for.lin.class} tells us that the set 
of $\msf{V}_{\pi/2}$ singularities of $\Phi$ has codimension $\leq k$ and satisfies \eqref{E:measure.ineq.for.lin.class}. (See section \ref{S:topology}.) 

\section{Linear Discriminant Analysis} \label{SS.lin.discr.anal}
We conclude this chapter by considering an important example. In linear discriminant analysis (LDA, Johnson and Wichern \cite[Chapter 11]{raJdwW92}) $\Phi$ is based on the sample covariance matrices of the positive and negative examples. (See \eqref{E:sample.covariance.matrix}.)  

The population covariance matrix of a random vector $X^{1 \times k}$ is the obvious analogue of the sample covariance matrix:
    \begin{equation*}
      \Sigma := E \bigl[ (X - EX)^{T} (X - EX) \bigr]^{k \times k} ,
    \end{equation*}
where ``$E$'' denotes expectation (integral w.r.t.\ probability measure). LDA is derived under the assumption that the positive and negative samples are each independent draws from (different) multivariate normal populations with the same, invertible \emph{population} covariance matrices. (But the population means, the $EX$'s, are not approximately equal. Otherwise linear classification is unpromising.) 

Suppose the data scientist training the classifier is willing to believe the preceding story is approximately correct. This might be a reasonable conclusion if the sample covariance matrices are approximately the same and invertible. In that case the data scientist might choose LDA.

(Unequal means is a safe assumption, but samples each independent draws from multivariate approximately normal populations? Approximately equal population covariance matrices? These questions are often answered by subjective judgments, i.e.\ non-algorithmically, after looking at the data. See section \ref{S:topology} and the last paragraph of section \ref{SS:general.lin.class.thy}.)

Let $P \in \NN_{n-1}$ and let $Y^{n \times \nvar} = (X, z^{P})$ be the data. 
Let $\bar{x}_{+}$ and $\bar{x}_{-}$ be the sample means of the predictors (rows of $X$) in the positive and negative samples. 
(Don't confuse $\bar{x}_{+}$ and $\bar{x}_{-}$ with the rows $x_{+,i}$, $x_{-,i}$, or $x_{\pm,i}$, of the $n \times k$ matrix $T = (T_{+}^{T}, T_{-}^{T})^{T}$! See \eqref{E:lin.class.T.matrix}.)  
Let $S_{+}$ be the function that computes the covariance matrix of the first $P$ rows of an $n \times k$ matrix. Define $S_{-}$ similarly. Let
    \begin{equation}  \label{E:S.pooled.defn}
      S_{pooled}^{k \times k} = S_{pooled}(X) := \mu S_{+}(X) + \nu S_{-}(X) , 
    \end{equation}
where $\mu, \nu > 0$ only depend on $P$ and $N$. (E.g., 
$\mu := (P-1)/\bigl[ (P-1) + (N-1) \bigr]$ and $\nu := 1 - \mu$, Johnson and Wichern \cite[Equation (11-17), p.\ 504]{raJdwW92}.) It follows from \eqref{E:sample.covariance.matrix} that $S_{+}(X)$, $S_{-}(X)$, and hence 
$S_{pooled}(X)$ are positive semidefinite, i.e., all their eigenvalues are nonnegative. 

We assume $S_{+}(X)$ and $S_{-}(X)$ are invertible. (A safe assumption unless two or more predictor variables are linearly related, a correctible problem.) Invertible positive semidefinite matrices are positive definite (i.e., have strictly positive eigenvalues) so $S_{pooled}$ is positive definite. 

The $\Phi$ corresponding to LDA is given by \eqref{E:LC.Phi} with $v \in S^{k-1}$ equal to 
    \begin{multline}  \label{E:Phi.LDA.defn}
      \Phi_{LDA}(Y)^{1 \times k}
        := |w|^{-1} w, \text{ where }
          w := (\bar{x}_{+} - \bar{x}_{-}) S_{pooled}^{-1}, \\
            \text{ whenever defined and nonzero} 
    \end{multline} 
(Johnson and Wichern \cite[Equation (11-18), p.\ 505]{raJdwW92}).  

Adopt this notation: If $M$ is a matrix 
    \begin{equation}  \label{E:LDA.overline}
      \text{let } \overline{M} 
        \text{ be the row vector consisting of the column means of } M .
    \end{equation} 
If $M$ has $n$ rows, let 
    \begin{equation*}
      M_{+} \text{ be the matrix consisting of the first $P$ rows of } M 
        \text{ and $M_{-}$ consist of the last } N
    \end{equation*}
Recall \eqref{E:zP.in.lin.class}. Let $\D$ be $\D_{\infty}$ or $\D_{\mu}$ defined in the section \ref{SS:general.lin.class.thy} and let 
    \begin{equation}  \label{E:LDA.D'.defn}
      \D' := \bigl\{ (X, z^{P}) \in \D : S_{pooled}(X) \text{ has full rank and} \
        \overline{X_{+}} \neq \overline{X_{-}} \bigr\} .
    \end{equation}
Then, by lemma \ref{L:rank.lwr.semicont},  
    \begin{equation}  \label{E:LDA.D'.is.nice}
      \Phi_{LDA} \text{ is continuous on $\D'$ and } \D' \text{ is open and dense in } \D .
    \end{equation}

Since \eqref{E:Phi.LDA.defn} is designed for the case of equal population covariance matrices, it simplifies things somewhat to choose $T_{+}^{P \times k}$ and 
$T_{-}^{(n-P) \times k}$, the partial data sets used in constructing 
$\T_{T}$, to have the same invertible covariance matrix, 
call it $\mbf{S}$. This means the numbers, $P$ and $N := n-P$, of positive and negative examples (the number of rows in $T_{+}$ and $T_{-}$, resp.) must each be at least $k+1$. I.e., \eqref{E:Phi.LDA.defn} only makes sense for $Y \in \T_{T}$ if 
    \begin{equation}  \label{E:in.LDA.P.N.geq.k+1}
      P, N \geq k+1 \text{ so } n \geq 2k+2 . 
    \end{equation}
(In particular, we assume the data, $Y$, are not ``wide'': $k \geq n$. See \eqref{E:n,k.in.lin.class}.)

In \eqref{E:S.pooled.defn}, take $\mu \in (0,1)$ and $\nu = 1- \mu$. Then 
    \begin{equation}  \label{E:LDA.cov.mat.on.T}
      \text{for } Y \in \T, \, S_{pooled} = \mbf{S} .
    \end{equation} 

Things are easier if the mean vectors, $\bar{T}_{+}$ and $\bar{T}_{-}$ resp. 
of $T_{+}$ and $T_{-}$ are chosen to both be 0. Let $s > 0$, $u \in S^{k-1}$, and $Y(u) = \bigl (X(u),z^{P} \bigr) \in \T$. Then, by \eqref{E:lin.class.X(u)} and \eqref{E:xi.in.lin.class}, 
    \begin{equation}  \label{E:bar.x.pm.=pm.su}
      \overline{X(u)_{+}} = su \text{ and } \overline{X(u)_{-}} = -su
    \end{equation} 
It follows from \eqref{E:LDA.cov.mat.on.T} and \eqref{E:Phi.LDA.defn} that 
$\T \subset \D'$. Moreover, $\Phi_{LDA}$ satisfies \eqref{E:LinClass.Theta.cont.on.T}. Therefore, by \eqref{E:LDA.D'.is.nice}, $\T$ has an open neighborhood lying in $\D'$. Moreover, \eqref{E:LDA.D'.is.nice} tells us $\Phi_{LDA}$ 
satisfies \eqref{E:Lin.Class.SV.pi/2.disjoint.from.T}.

Define $\delta$ as in \eqref{E:lin.class.delta.defn}. Let $\lambda_{min}$ be the smallest eigenvalue of $\mbf{S}^{-1}$. Thus, $\lambda_{min} > 0$. Let $\lambda_{max} \geq \lambda_{min}$ be the largest eigenvalue of $\mbf{S}^{-1}$. \emph{Claim:} 
$\Phi = \Phi_{LDA}$ satisfies \eqref{E:Theta(Y).dot.Lambda(Y).big} if  
$s > 2 \delta \lambda_{max}/\lambda_{min}$. By \eqref{E:bar.x.pm.=pm.su}, 
   \begin{equation}  \label{E:overline.X(u).diff}
      \overline{X(u)_{+}} - \overline{X(u)_{-}} = 2 s u . 
   \end{equation}

Thus, $\T \subset \D'$ as defined in \eqref{E:LDA.D'.defn}. Thus, $\Phi_{LDA}$ is defined everywhere on $\T$ and, by \eqref{E:LDA.D'.is.nice}, is continuous on $\T$ as well. A consequence is that \eqref{E:LinClass.Theta.cont.on.T} holds with 
$\Theta = \Phi_{LDA} \restriction_{\T}$. 

By \eqref{E:overline.X(u).diff}, we have 
    \begin{equation}  \label{E:X.bar.diff.S.invrs.lambda} 
      \Bigl| \bigl( \, \overline{X(u)_{+}} - \overline{X(u)_{-}} \,  \bigr) \mbf{S}^{-1} \Bigr|
       = 2 s | u \mbf{S}^{-1} | \leq 2 s \lambda_{max} . 
    \end{equation} 
Suppose 
    \begin{equation*}
      s > 2 \delta \lambda_{max}/\lambda_{min} .
    \end{equation*} 
By \eqref{E:xi.in.lin.class}, the first row of $X$ is $x_{1} := x_{+,1} + z_{i} s u$. Thus, by 
\eqref{E:Phi.LDA.defn}, \eqref{E:overline.X(u).diff}, \eqref{E:LDA.cov.mat.on.T},  and \eqref{E:X.bar.diff.S.invrs.lambda}, 
    \begin{multline}  \label{E:Phi.dot.Lamdbda.for.LDA}
      \Theta \bigl( Y(u) \bigr) \cdot \Lambda \bigl( Y(u) \bigr) 
          = \Bigr( \bigl|( \, \overline{X(u)_{+}} - \overline{X(u)_{-}} \, ) \mbf{S}^{-1} \bigr|^{-1} 
                      ( \, \overline{X(u)_{+}} - \overline{X(u)_{-}}) \mbf{S}^{-1} \Bigl) \cdot u \\
            = \Bigr( 2 \bigl| ( \, \overline{X(u)_{+}} - \overline{X(u)_{-}} \, ) 
              \mbf{S}^{-1} \bigr|^{-1} s u \mbf{S}^{-1} \Bigl) \cdot u 
                \geq \frac{1}{2 s \lambda_{max}} 2 s (u \mbf{S}^{-1} \cdot u) 
                  \geq \frac{\lambda_{min}}{\lambda_{max}} .
    \end{multline}
Thus, if $s > 2 \delta \lambda_{max}/\lambda_{min}$ we have 
$\Theta \bigl( Y(u) \bigr) \cdot \Lambda \bigl( Y(u) \bigr) > 2 \delta/s > 0$, as claimed. 
Hence, by \eqref{E:big.Phi.dot.Lambda.big.Gamma.correct.on.T}, if 
$s > 2 \delta \lambda_{max}/\lambda_{min}$ we can choose a 
$b \in \RR$ s.t.\ LDA correctly classifies every observation in $Y$ for every 
$Y \in \T$.

\eqref{E:Phi.dot.Lamdbda.for.LDA} implies \eqref{E:Theta.dot.Lambda.>.-1} 
for $\Phi_{LDA}$. We have already seen that LDA satisfies 
\eqref{E:Lin.Class.SV.pi/2.disjoint.from.T}. Therefore proposition \ref{P:measure.bnd.for.lin.class} holds for LDA.

Until $\gamma$ in \eqref{E:measure.ineq.for.lin.class} has been determined, that inequality is useless for a specific linear classifier like LDA. The best we can do is establish \eqref{E:sing.codim.in.lin.class} for it. We do that now. We have defined $\X_{n}^{k}$, 
$\X_{\infty}$, and $\X_{\mu}$. To add to the confusion, define
    \begin{equation} \label{E:LDA.X0.defn}
      \X_{0} := \{ X \in \X_{n}^{k} : \overline{X_{-}} = \overline{X_{+}} \} .
    \end{equation}
So $\X_{0}$ is a linear subspace of $\X_{n}^{k}$. Define 
    \begin{equation}  \label{E:LDA.S0.defn}
      \Ss_{0} := \X_{0} \times \{z^{P}\} . 
    \end{equation}
of $\X_{n}^{k} \times \{z^{P}\}$.) $\Ss_{0}$ is isometric to $\X_{0}$ so that the projection 
$\Ss_{0} \to \X_{0}$ preserves Hausdorff dimension and measure. 

We make the non-astonishing \emph{claim:}
    \begin{equation}  \label{E:Ss0.subset.SS.V.pi/2}
      \Ss_{0} \subset \Ss^{\msf{V}_{\pi/2}} .
    \end{equation}
So, by \eqref{E:LDA.D'.defn}, $\Ss_{0} \cap \D' = \varnothing$. 

Let $(X, z^{P}) \in \Ss_{0}$. We do not require $S_{pooled}(X)$ be non-singular. 
We begin by constructing an $n \times k$ matrix $\tilde{X}$ approximating $X$ but s.t.\ $S_{pooled}(\tilde{X})$ has full rank $k$. If $S_{pooled}(X)$ has full rank $k$, then just 
let $\tilde{X} := X$. 
    \begin{equation}  \label{E:X.tilde.when.S.has.full.rank}
      \text{If $S_{pooled}(X)$ has full rank $k$, then let } \tilde{X} := X .
    \end{equation}

Suppose $rank \, S_{pooled}(X) < k$. Define the ``null space'', 
$null \, R$, of a $k \times k$ matrix $R$ to be the linear subspace, 
$null \, R := \{ x^{1 \times k} \in \RR^{k} : x R = 0 \}$. Let 
    \begin{equation*}
      m := \dim \bigl( null \, S_{pooled}(X) \bigr) > 0
        \text{ and } \ell := k-m < k . 
    \end{equation*}
So $\ell = rank \, S_{pooled}(X)$. Similarly, let 
    \begin{equation}  \label{E:ell+,m+,rank.S+}
      m_{+} := \dim \bigl( null \, S_{+}(X)  \bigr) \text{ and } 
        \ell_{+} := k-m_{+} = rank \, S_{+}(X) .
    \end{equation} 
Notice that by \eqref{E:S.pooled.defn} and the fact that covariance matrices are non-negative definite, \linebreak 
$null \, S_{pooled}(X) = \bigl( null \, S_{+}(X) \bigr) \cap \bigl( null \, S_{-}(X) \bigr)$. 
Therefore, 
    \begin{equation*}
      m_{+} \geq m \text{ so } \ell_{+} \leq \ell < k. \text{ Hence, } \ell_{+} < k .
    \end{equation*}
    
$S_{+}(X)$ is symmetric. We can construct a matrix $C_{+}^{m_{+} \times k}$ with orthonormal rows s.t.\ 
    \begin{equation}  \label{E:C+.S+=0}
      C_{+} S_{+}(X)^{T} = C_{+} S_{+}(X) = 0 . 
    \end{equation}
Since $m_{+} \leq k$ we have $rank \, C_{+} = m_{+}$. The rows of $C_{+}$ form an orthonormal basis of $null \, S_{+}(X)$. Eigenvectors 
of $C_{+}^{T} C_{+}$ are the rows of $C_{+}$ (eigenvalues = 1) plus any $k-\ell_{+}$ orthonormal vectors orthogonal to the the rows of $C_{+}$ (eigenvalues = 0). 

If $\ell_{+} = 0$ then $m_{+} = k$ so $S_{+}(X) = 0$,  $C_{+}$ is $k \times k$, and by \eqref{E:in.LDA.P.N.geq.k+1} there is a $P \times k = P \times (k-\ell_{+})$ matrix $M$ with orthonormal columns s.t.\ $(1^{P})^{T}  M = 0$. Since $S_{+}(X) = 0$ we must have $X_{+} = 1^{P} x_{1}$, where $x_{1}$ is a $k$-dimensional row vector. Thus,
    \begin{equation}  \label{E:MX+=0.if.ell+=0}
      M^{T} X_{+} = 0 \text{ if } \ell_{+} = 0 .
    \end{equation}  

Suppose $\ell_{+} \in (0,k)$. Let 
    \begin{equation*}
      X_{+0}^{P \times k} := X_{+} - 1^{P}  \overline{X_{+}} .
    \end{equation*} 
By \eqref{E:ell+,m+,rank.S+}, $rank \, S_{+}(X) = \ell_{+}$. Thus, by the Singular Value Decomposition, Rao \cite[(v), p.\ 42]{crR73.LinStatInf}, we may write 
    \begin{equation*}
      X_{+0} = (A^{P \times \ell_{+}})( B^{\ell_{+} \times k}),
    \end{equation*} 
where the rows of $B$ are orthonormal. Recall \eqref{E:rho=row.space}. 
We have $\rho(B) = \rho(X_{+0})$. Since $(1^{P})^{T} X_{+0} = 0$, we have 
    \begin{equation*}
      (1^{P})^{T} A = 0 .
    \end{equation*} 

Now, by \eqref{E:sample.covariance.matrix}, 
    \begin{equation}  \label{E:S+(X) = BT.AT.A.B}
      S_{+}(X) =  \frac{1}{P-1} X_{+0}^{T} \, X_{+0} = B^{T} A^{T} A B. 
    \end{equation}
Since $rank \, S_{+}(X) = \ell_{+}$, we have $(B^{T} A^{T} A B)^{k \times k}$ has rank $\ell_{+}$. \eqref{E:S+(X) = BT.AT.A.B} also implies $B^{T} S_{+}(X) B = A^{T} A$. Therefore, $rank \, A = \ell_{+}$ as well. 

Let $W$ be the space of $w^{P \times 1}$ s.t.\ $w^{T} A = 0$. $W$ has dimension 
$P - \ell_{+}$. By \eqref{E:in.LDA.P.N.geq.k+1} and our current assumption that 
$\ell_{+} \in (0,k)$, we have $P - \ell_{+} > P-k > 0$. Requiring $w^{T} 1^{P} = 0$ removes another degree of freedom: $(1^{P})^{T} A = 0$ but we are now going to exclude it from $W$.  

So the space, $W_{1} \subset W$, of $P$-dimensional column vectors $w$ s.t.\ 
$w^{T} A = 0$ and $w^{T} 1^{P} = 0$ has dimension 
$P - \ell_{+} -1 = (P-k-1) + (k-\ell_{+})$. Therefore, by \eqref{E:in.LDA.P.N.geq.k+1} again and our current assuption that $\ell_{+} \in (0,k)$, we have 
$P - \ell_{+} -1 \geq k-\ell_{+} > 0$. 

Hence, if $0 < \ell_{+} < k$, we can construct a 
$P \times (k-\ell_{+})$ matrix $M$ with orthonormal columns in $W_{1}$ s.t.\ 
$M^{T} A = 0$ and $(1^{P})^{T} M = 0$. Thus, we have 
    \begin{equation*}
      M^{T} X_{+} = M^{T} (X_{+0} + 1^{P} \overline{X_{+}}) 
        = M^{T} AB + M^{T} 1^{P} \overline{X_{+}} = 0  \text{ if } \ell_{+} \in (0,k) .  
    \end{equation*}
Combining this with \eqref{E:MX+=0.if.ell+=0} we get
    \begin{multline}  \label{E:MT.X+=0}
       \text{If } 0 \leq \ell_{+} < k \text{ then } 
         M \text{ is a } P \times (k-\ell_{+}) \\
           \text{ matrix with orthonormal columns and } M X_{+} = 0 .
    \end{multline} 

Suppose $0 \leq \ell_{+} < k$ and let $\epsilon > 0$ be arbitrary. 
Let 
    \begin{equation}  \label{E:lin.disc.anal.X+tilde}
      \widetilde{X_{+}} := X_{+} + \epsilon M C_{+} .
    \end{equation}  
Since $(1^{P})^{T}  M = 0$, the vector of column means of $\widetilde{X_{+}}$ is the same as that of $X_{+}$, viz.\ $\overline{X_{+}}$. Hence, by \eqref{E:sample.covariance.matrix} and \eqref{E:MT.X+=0}, 
        \begin{align*}
          (P-1) S(\widetilde{X_{+}}) 
            &= (\widetilde{X_{+}})^{T} \widetilde{X_{+}} - P \, (\overline{X_{+}})^{T} 
              \, \overline{X_{+}} \\
            &= (X_{+} + \epsilon M C_{+})^{T}  (X_{+} + \epsilon M C_{+})  
              - P \, (\overline{X_{+}})^{T} \, \overline{X_{+}}  \\
            &= X_{+}^{T} X_{+} + \epsilon C_{+}^{T} M^{T} X_{+} 
              + \epsilon X_{+}^{T} MC_{+}
                + \epsilon^{2} C_{+}^{T} M^{T} M C_{+} - P \, 
                  (\overline{X_{+}})^{T} \overline{X_{+}}  \\
            &= \bigl( X_{+}^{T} X_{+} - P \, (\overline{X_{+}})^{T} \; \overline{X_{+}} \bigr) 
              + 0 + 0 + \epsilon^{2} C_{+}^{T} M^{T} M C_{+}  \\
            &= (P-1) S_{+}(X) + \epsilon^{2} C_{+}^{T} C_{+} .
        \end{align*}
Recall \eqref{E:ell+,m+,rank.S+}. By \eqref{E:C+.S+=0} and the fact that 
$rank \, C_{+} = m_{+} = k-\ell_{+} = k - \bigl[ rank \, S_{+}(X) \bigr]$, 
we have that $S(\widetilde{X_{+}})$ has full rank.

Let $\widetilde{X}^{n \times k} = \widetilde{X}_{\epsilon}$ be the result of stacking $\widetilde{X_{+}}$ on top of $X_{-}$. Thus, $S_{+}(\widetilde{X}) = S(\widetilde{X_{+}})$ and $S_{-}(\widetilde{X}) = S_{-}(X)$. In addition, by \eqref{E:S.pooled.defn}, if $0 \leq \ell_{+} < k$, 
    \begin{multline*}
      S_{pooled}(\widetilde{X}) = \mu S_{+}(\widetilde{X}) + \nu S_{-}(\widetilde{X}) 
        = \mu S_{+}(X) + \nu S_{-}(X) + \mu \frac{\epsilon^{2}}{P-1} C_{+}^{T} C_{+} \\
          = S_{pooled}(X) + \frac{\epsilon^{2} \mu}{P-1} C_{+}^{T} C_{+} .
    \end{multline*}
Since $S(\widetilde{X_{+}})$ has full rank so does $S_{pooled}(\widetilde{X})$. (By \eqref{E:X.tilde.when.S.has.full.rank}, if $S_{pooled}(X)$ has full rank $k$, 
then $\widetilde{X} := X$. So in this case again $S_{pooled}(\widetilde{X})$ has full rank.)

Let $u \in S^{k-1}$ (written as a row vector) and let 
    \begin{equation}  \label{E:lin.disc.anal.y}
      y^{1 \times k} := u \, S_{pooled}(\widetilde{X}) \neq 0 .
    \end{equation} 
Let $(1^{P,n})^{T} := (+1, \ldots, +1, -1, \ldots, -1)^{1 \times n}$, where the number 
of ``$+1$'s'' is $P$ and the number of ``$-1$'s'' is $n-P$. (Thus, $1^{P,n}$ is 
$n \times 1$. Note that $1^{n,n} = 1^{n}$ as in \eqref{E:1n.col.vec.defn}. 
Yes, $1^{P,n} = z^{P}$ defined in \eqref{E:zP.in.lin.class}, but to prevent confusing the ``$X$'' realm with the ``$z$'' realm here I call it $1^{P,n}$ instead.)  
Let $X' := X'_{\epsilon} := \widetilde{X} + \epsilon 1^{P,n} y$. Then, by decreasing 
$\epsilon$, $X'_{\epsilon}$ can be made arbitrarily close to $X$. $X'$ is obtained 
from $\widetilde{X}$ by adding the same vector, $\epsilon y$ to each of its first $P$ rows and subtracting $\epsilon y$ from each of its last $n-P$ rows. 
The $\epsilon y$ cancels in the process of computing $S_{+}(X')$ 
and $S_{-}(X')$. Thus, $S_{+}(X') = S_{+}(\widetilde{X})$ and 
$S_{-}(X') = S_{-}(\widetilde{X})$. 
Therefore, $S_{pooled}(X') = S_{pooled}(\widetilde{X})$ so $S_{pooled}(X')$ is invertible. It has full rank.
Let $\bar{x}'_{+}$ and $\bar{x}'_{-}$ be the means of the first $P$ and last $N$ rows of $X'$, resp. Thus, by definition, \eqref{E:LDA.X0.defn}, of $\X_{0}$, 
$\bar{x}'_{+} - \bar{x}'_{-} = (\overline{X_{+}} + \epsilon y) - (\overline{X_{-}} 
- \epsilon y) = 2 \epsilon y \neq 0$. And, as just mentioned, $S_{pooled}(X')$ has full rank. Therefore, by \eqref{E:LDA.D'.defn}, $X' \in \D'$. Moreover, by \eqref{E:Phi.LDA.defn} and \eqref{E:lin.disc.anal.y}, 
    \begin{equation*}
      \Phi_{LDA}(X'_{\epsilon}) \propto 2 \epsilon y S_{pooled}(X'_{\epsilon})^{-1} 
        = 2 \epsilon u \, S_{pooled}(\widetilde{X}) S_{pooled}(\widetilde{X})^{-1} 
          = 2 \epsilon u \neq 0 
    \end{equation*}
By \eqref{E:LC.Phi}, we have $\bigl| \Phi_{LDA}(X'_{\epsilon}) \bigr| = 1$. Therefore, in fact,  
$\Phi_{LDA}(X'_{\epsilon}) = u$. Since $u \in S^{k-1}$ is arbitrary, we see that 
$X \in \Ss^{\msf{V}_{\pi/2}}$. This proves the claim \eqref{E:Ss0.subset.SS.V.pi/2}.

Remember $\Ss_{0}$ defined in \eqref{E:LDA.S0.defn}? We have just proved \eqref{E:Ss0.subset.SS.V.pi/2}, 
$\Ss_{0} \subset \Ss^{\msf{V}_{\pi/2}}$. What is the dimension of $\Ss_{0}$? 

Let $X \in \X_{n}^{k}$ (see \eqref{E:mcl.Xmk.lin.class.defn}) and regard $X$ as an $nk$-dimensional column vector by concatenating its columns in order. 
Denote that vector by $c(X)^{nk \times 1}$. $\X_{0}$ (see \eqref{E:LDA.X0.defn}) is the kernel of the linear transformation 
$B : c(X) \mapsto (\overline{X_{+}} - \overline{X_{-}})^{T} \in \RR^{k}$. (By \eqref{E:LDA.overline}, $(\overline{X_{+}} - \overline{X_{-}})^{T}$ 
is a $k \times 1$ column vector.) Let $E$ be of the matrix of $B$ (applied on the left). $E$ has dimensions $k \times nk$. 
Let $\eta^{1 \times n} := (P^{-1} (1^{P})^{T} \, , \, -N^{-1} (1^{N})^{T})$. So $\eta$ is the 
$1 \times n$ row matrix obtained by concatenating $P^{-1} (1^{P})^{T}$ followed by 
$-N^{-1} (1^{N})^{T}$. Let $O^{1 \times n}$ have all 0 entries. Then it is easy to see that
    \begin{equation*}
      E^{k \times nk} =
        \begin{pmatrix}
            \eta & O & O & \ldots & O & O \\
            O & \eta & O & \ldots & O & O \\
            \vdots & & \vdots & & & \vdots \\
            O & O & O & \ldots & O & \eta
        \end{pmatrix} .
    \end{equation*}
I.e., if $X \in \X_{n}^{k}$ then $E \, c(X) = (\overline{X_{+}} - \overline{X_{-}})^{T}$. Clearly, 
$rank \, E = k$. Therefore, by Stoll and Wong \cite[Theorem 2.1, p.\ 99]{rrSetW68.LinearAlgebra}, the dimension of the kernel, $\Ss_{0}$, of that transformation is $nk - k$. Hence, \emph{as a subset of} 
$\X_{n}^{k} \times \{z^{P}\}$ we have $\dim \Ss_{0} = nk - k = d-k$.  
So $\text{codim} \, \Ss_{0} = k$ (in $\X_{n}^{k} \times \{z^{P}\}$) consistent with  \eqref{E:sing.codim.in.lin.class}.

But $\X_{n}^{k} \times \{z^{P}\}$ is not $\D$.  Recall \eqref{E:D.X.options.in.lin.class}. Identify $\D$ with $\X$. Start with \eqref{E:LDA.D.infty.defn}, and \eqref{E:LDA.X.infty.defn}. Suppose $\D = \D_{\infty}$. Then $\X$ is correspondingly the one point compactification, $\X_{\infty}$, of $\RR^{nk}$. 
As in a similar discussion in section \ref{SS:D.T.plane.fit}, the inclusion 
$\X_{n}^{k} \hookrightarrow \X_{\infty}$ is locally bi-Lipschitz onto its image 
$\X_{\infty} \setminus \{\infty\}$. Therefore, by lemma \ref{L:loc.Lip.image.of.null.set.is.null}, as a subset of $\X_{\infty}$ the codimension of $\X_{0}$ is also $k$. 

As for $\Hm^{d-k}(\Ss_{0})$, $\Ss_{0}$ is a $(nk - k)$-dimensional linear subspace 
of $\RR^{nk}$. Therefore, by theorem \ref{T:Hausdorff.measure.gives.right.value.on.manifs}, 
$\Hm^{nk-k}(\Ss_{0}) = \infty$, where $\Hm^{nk-k}(\Ss_{0})$ is computed w.r.t.\ the Euclidean metric on $\RR^{nk}$. Define the inverse stereographic projection map 
$PS : \RR^{nk} \to S^{nk}\bigl( (0^{1 \times n k}, 1) \bigr)$ (the unit sphere 
in $\RR^{nk+1}$ with center at $(0^{1 \times n k},1)$) as in \eqref{E:PS.formla}. 
As in section \ref{SS:D.T.plane.fit} $S^{nk}\bigl( (0^{1 \times n k}, 1) \bigr)$ is the concrete interpretation of $\X_{\infty}$ and hence of $\D_{\infty}$. Let $S'$ be 
$\X_{\infty} = S^{nk}\bigl( (0^{1 \times n k}, 1) \bigr)$ with its north pole, 
$(0^{1 \times n k}, 2)$,  removed. 
$PS$ is locally Lipschitz on $\RR^{nk}$ and its inverse (stereographic projection) is locally Lipschitz on $S'$. Let $\mcl{A}$ be a compact subset of the linear space 
$\Ss_{0} \subset \RR^{nk}$. Suppose $\Hm^{nk-k}(\mcl{A}) > 0$, computed based on the Euclidean metric on $\Ss_{0}$. It follows from \eqref{E:local.Lip.is.Lip.on.compacts} and \eqref{E:Lip.magnification.of.Hm} that 
$\Hm^{nk-k}\bigl[ PS(\mcl{A}) \bigr] > 0$, computed w.r.t.\ the manifold metric on $S'$ is inherits from $\RR^{nk+1}$. Therefore, since 
$\X_{\infty} = S^{nk}\bigl( (0^{1 \times n \nvar}, 1) \bigr)$, we see that 
$\Hm^{nk-k}\bigl[ PS(\Ss_{0}) \bigr] = \Hm^{nk-k}\bigl[ PS(\X_{0}) \bigr] > 0$. We have thus shown that \eqref{E:Ss0.subset.SS.V.pi/2} holds for LDA when $\D = \D_{\infty}$.  

Now recall \eqref{E:LDA.D.mu.defn} and \eqref{E:LDA.X.mu.defn}. If $\D = \D_{\mu}$, we need to analyze $\X_{0} \cap \X_{\mu}$ (see \eqref{E:LDA.X0.defn}). First, suppose 
$\mu : \RR^{nk} \to \RR$ is constant and positive. Then $\X_{\mu}$ is an $(nk-1)$-sphere of radius $\mu$ centered at the origin. Since $\X_{0}$ is a $(nk - k)$-dimensional linear subspace of $\RR^{nk}$, we have that $\X_{0} \cap \X_{\mu}$ is a $\dim \X_{0} - 1$ dimensional sphere of radius $\mu$. I.e., 
$\dim (\D_{0} \cap \D_{\mu}) = \dim \X_{0} - 1 = (\dim \RR^{nk} - k) - 1 = (nk - 1) - k 
= d - k = \dim(\D) - k$. (See  \eqref{E:LinClass.D.mu.diffeom.to.sphere}.) Of course, by \eqref{E:Lip.magnification.of.Hm} and theorem \ref{T:Hausdorff.measure.gives.right.value.on.manifs} again, 
$\Hm^{d-k}(\Ss_{0} \cap \D_{\mu}) =  \mu^{d-k} \, \Hm^{d-k}(S^{d-k}) > 0$, 
where $\Hm^{d-k}$ is computed w.r.t.\ the manifold metric on 
$\Ss_{0} \cap \D_{\mu}$ it inherits from $\RR^{nk}$.

Next, let $\mu : S^{n k-1} \to (0, \infty)$ be any smooth map. $S^{n k-1}$ is mapped 
onto $\X_{\mu}$ by the obvious analogue of $R_{\mu}$ defined in \eqref{E:R.mu.defn}. Specifically, $X \in \X_{n}^{k} \setminus \{0\}$ is mapped 
to $r_{\mu}(X) := \mu \bigl( \| X \|^{-1} X \bigr) \| X \|^{-1} X$. As in \eqref{E:D.mu.diffeo.to.sphere} with $k$ in place of $\nvar$, the restriction
$r_{\mu} \restriction_{S^{n k-1}} : S^{n k-1} \to \D_{\mu}$ is a diffeomorphism. 

Let $X \in \X_{0}$. Then $r_{\mu}(X)$ is just a multiple of $X$. Therefore, by \eqref{E:LDA.X0.defn}, $r_{\mu}(X) \in \X_{0}$, too. And conversely. I.e., 
Now, $r_{\mu}$ maps $\X_{0} \cap S^{nk-1}$ onto $\X_{0} \cap \X_{\mu}$. 
By corollary \ref{C:cont.diff.=.loc.Lip}, $r_{\mu} \restriction_{S^{n k-1}}$ is bi-locally Lipschitz. In fact, by compactness and \eqref{E:local.Lip.is.Lip.on.compacts}, 
$r_{\mu} \restriction_{S^{n k-1}}$ is bi-Lipschitz. Hence, by  
\eqref{E:Lip.magnification.of.Hm} and the last paragraph, 
    \begin{equation*}
      \text{codim} \, (\Ss_{0} \cap \D_{\mu}) = k \text{ and }
        \Hm^{d-k}(\Ss_{0} \cap \D_{\mu}) > 0 .
    \end{equation*} 

By \eqref{E:Ss0.subset.SS.V.pi/2}, this confirms \eqref{E:sing.codim.in.lin.class} for linear discriminant analysis.     
  
\chapter{Summary, Discussion, Conclusions}   \label{Chptr:sum.dis.con}
The generality of our results in chapters \ref{Chptr:topology}, \ref{Chptr:Haus.meas.of.sing.set}, and \ref{Chptr:severity} together with the broad range of examples 
(chapters \ref{Chptr:sings.in.plane.fit}, \ref{Chptr:spherical.location}, \ref{Chptr:aug.direct.mean},  \ref{Chptr:robst.loc.on.circle}), and \ref{Chptr:linear.classification} indicate that the singularity problem of data maps is a widespread problem. 

This book concerns data analytic \emph{functions}, not just algorithms, because it applies to biological, not just formal, cognition.

The main remaining question is, what is the impact of singularities in practice?  If a function has singularities, then one must consider what the probability is that one will get data near a singularity (subsection \ref{SS:asymp.prob} and corollary \ref{C:lwr.bound.on.prob.of.being.close.to.S'}). 

Remark \ref{R:mitigating.sings} offers suggestions for how to deal with singularity.

\clearpage

\begin{center}
\begin{Large}
APPENDICES
\end{Large}
\end{center}

\setcounter{chapter}{0}
\setcounter{section}{0}
\setcounter{equation}{0}
\renewcommand{\thechapter}{\Alph{chapter}}
\renewcommand{\chaptername}{APPENDIX}
\renewcommand{\thesection}{\thechapter.\arabic{section}}
\numberwithin{equation}{section}
\numberwithin{theorem}{section}
\renewcommand{\thetheorem}{\thechapter.\arabic{theorem}}

\chapter{Some Technicalities}  \label{Chptr:misc.proofs}
  
  \begin{proof}[Proof of lemma \ref{L:extend.Phi.to.D.less.S}]
It is immediate from \eqref{E:D'.dense.Phi.cont.on.D'} that $\Ss$ has empty interior. Next, we prove that 
$\hat{\Phi}$ is continuous on $\D \setminus \Ss$. Let $x_{0} \in \D \setminus \Ss$. We will show that 
$\hat{\Phi}$ is continuous at $x_{0}$. If $x_{0}$ is isolated there is nothing to prove. So assume $x_{0}$ is not isolated and let $V \subset \F$ be a neighborhood of $\hat{\Phi}(x_{0})$. By \eqref{E:D.metric.F.normal}, there exists an open neighborhood, $W$, of $\hat{\Phi}(x_{0})$ s.t.\ $\overline{W} \subset V$. By definition of $\hat{\Phi}(x_{0})$, we may pick a neighborhood, $\clU$ of $x_{0}$ s.t.\ $x \in \clU \cap \D'$ implies 
$\Phi(x) \in W$. 
It follows that if $x \in \clU \setminus \Ss$ then $\hat{\Phi}(x) \in \overline{W} \subset V$, proving continuity at $x_{0}$. Thus, the singular set of $\hat{\Phi}$ w.r.t.\ $\D \setminus \Ss$ is a subset of $\Ss$. Since \eqref{E:D.metric.F.normal} and \eqref{E:D'.dense.Phi.cont.on.D'} imply that $\hat{\Phi} \restriction_{\D'} = \Phi \restriction_{\D'}$, every point of $\Ss$ is a singularity of $\hat{\Phi}$ w.r.t.\ 
$\D \setminus \Ss$.

Suppose $\tilde{\D}$ is another dense subset of $\D$ on which $\hat{\Phi}$ is defined, and let $\tilde{\Ss}$ be the singular set of $\hat{\Phi}$ w.r.t.\ $\tilde{\D}$. 
Now, $\hat{\Phi}$ is continuous 
on $\D \setminus \Ss \supset \tilde{\D}$. So any singularity of $\hat{\Phi}$ 
w.r.t.\ $\tilde{\D}$ must be in $\Ss$.
  \end{proof}

  \begin{proof}[Proof of lemma \ref{L:data.maps.defined.by.opt}]
First, we prove part \eqref{I:unif.opt}. Suppose the hypotheses of part \eqref{I:unif.opt} hold and let $x' \in \D'_{\ref{I:unif.opt}}$. We show that $\Phi(x')$ exists. Let $f_{0} \in F$ be as in \eqref{E:outside.G.g.is.big.unif}, but suppose $f_{1} \in \F \setminus \{ f_{0} \}$ satisfies $g(f_{1}, x') \leq g(f_{0}, x')$. Since, by \eqref{E:D.metric.F.normal}, $\F$ is normal, there exists a neighborhood, $G$ of $f_{0}$ s.t.\ $f_{1} \notin G$. Then, by \eqref{E:outside.G.g.is.big.unif}, $g(f_{1}, x') > g(f_{0}, x')$. Contradiction. Thus, $\Phi(x')$ exists uniquely and equals $f_{0}$.

Suppose the hypotheses of part \eqref{I:unif.opt} hold, but $\Phi$ is \emph{not} continuous on $\D'_{\ref{I:unif.opt}}$. Then there exists $x' \in \D'_{\ref{I:unif.opt}}$ and a sequence $\{ x_{n} \} \subset \D'_{\ref{I:unif.opt}}$ s.t.\ $x_{n} \to x'$, 
but $\Phi(x_{n}) \nrightarrow f_{0} := \Phi(x')$. (By \eqref{E:D.metric.F.normal}, $\D$ is a metric space, hence, first countable.)  Taking a subsequence if necessary, we may assume that there is a neighborhood, $G \subset \F$, of $f_{0}$ s.t.\ for no $n$ do we have $\Phi(x_{n}) \in G$.  
By \eqref{E:outside.G.g.is.big.unif}, there exists a neighborhood $\clU \subset \D$ of $x'$ s.t.\ $g(f_{0}, x) < \inf_{f \notin G} g(f, x)$ for every $x \in \clU$. But eventually $x_{n} \in \clU$. Therefore, eventually, 
	\[
	  g \bigl[ \Phi(x_{n}), x_{n} \bigr] < g(f_{0}, x_{n}) 
	    < g \bigl[ \Phi(x_{n}), x_{n} \bigr].
	\]
Contradiction. 

Next, suppose the hypotheses of  part \eqref{I:cmpct.opt} hold but there exists $x' \in \D'_{1}$ s.t.\ $\Phi$ is not continuous at $x'$. Then, there is a sequence $\{ x_{n} \} \subset \D'_{1}$ and a neighborhood, $G \subset \F$, of $f_{0} := \Phi(x')$ s.t.\ $x_{n} \to x'$  and $\Phi(x_{n}) \in \F \setminus G$ for every $n$. Therefore, 
for every $n = 1, 2, \ldots$, 
	\[
			g \bigl[ \Phi(x_{n}), x_{n} \bigr] < g (f_{0}, x_{n} ).
	\]
By compactness of $\F$ we may assume $\Phi(x_{n})$ converges to $f_{\infty} \in \F \setminus G$, say. Thus, $f_{\infty} \neq f_{0}$. Since $g$ is continuous, we then have
	\[
		g(f_{\infty}, x') \leq g (f_{0}, x' ). 	\]
Hence, $f \mapsto g(f, x')$ does not have a unique minimum at $f_{0}$, contradicting the assumption that $x' \in \D'_{\ref{I:cmpct.opt}}$ and $f_{0} = \Phi(x')$. 
  \end{proof}

  \begin{proof}[Proof of proposition \ref{P:tube.about.sing.set}]
 \emph{Claim:} There exist constants $0 < a < b < \infty$ depending only on $\D$ and $\Rcl$ s.t.\  for $\eta > 0$  sufficiently small
   \begin{equation}  \label{E:vol.of.ball.approx.by.power}
      a \eta^{d} \leq \Hm^{d} \bigl( \mcl{B}_{\eta}(x) \bigr) \leq b  \eta^{d}, 
          \text{ for every } x \in \Rcl,
   \end{equation}
where $\mcl{B}_{\eta}(x) \subset \D$ is the open ball with center at $x$ and radius $\eta$. 
To prove this, first we prove
   \begin{multline}  \label{E:finite.numbr.coord.nbhds.cover.any.ball}
      \text{There exist coordinate neighborhoods, } (\clU_{i}, \varphi_{i}), 
        i = 1, \ldots, N, \text{ in } \D \\
         \text{ with } \varphi_{i} : \clU_{i} \to \RR^{d} \text{ and its inverse Lipschitz,  } 
            (i = 1, \ldots, N) \text{ s.t.\ for $\eta$ sufficiently small, } \\ 
              \text{ for every } x \in \Rcl \text{ there exists $i = 1, \ldots, N$ s.t.\ } 
                \mcl{B}_{\eta}(x) \subset \clU_{i}.
   \end{multline}

To prove \eqref{E:finite.numbr.coord.nbhds.cover.any.ball} we in turn first prove that there exists $\mcl{C} \subset \D$, compact, s.t.\ $\Rcl \subset \mcl{C}^{\circ}$, the interior of $\mcl{C}$. To this end, note that we may assume that any $x \in \D$ has a coordinate neighborhood 
$(\clU, \varphi)$ in $\D$ with the following properties.
	\begin{enumerate}
		\item $\varphi(\clU)$ is bounded.
		\item $\varphi$ extends to the closure, $\overline{\clU}$ and $\varphi^{-1}$ extends to the closure, $\overline{\varphi(\clU)}$.
		\item $\varphi$ and $\varphi^{-1}$ are Lipschitz on $\overline{\clU}$ and $\overline{\varphi(\clU)}$, resp.
	\end{enumerate}
Since $\Rcl$ is compact it has a finite covering $(\clU_{i}, \varphi_{i})$ ($i=1, \ldots, N$) consisting of coordinate neighborhoods with the preceding three properties. Let $\mcl{C} := \bigcup_{i} \overline{\clU_{i}}$.

Since $\mcl{C}$ is compact and $\D$ is a Lipschitz manifold, there exists $\eta_{0} > 0$ s.t.\ if $0 < \eta \leq \eta_{0}$ then for any 
$x \in \Rcl$ there exists a coordinate neighborhood $(\clU_{x}, \varphi_{x})$
(with $\varphi_{x}$ and $\varphi_{x}^{-1}$ Lipschitz) s.t.\ $\mcl{B}_{2 \eta}(x) \subset \clU_{x}$ (Lebesgue's covering lemma, Simmons \cite[Theorem C, p.\ 122]{gfS63}). Choose $x_{1}, \ldots, x_{M} \in \Rcl$ 
s.t.\ $\Rcl \subset \bigcup_{i=1}^{M} \mcl{B}_{\eta_{0}}(x_{i})$.
Then $(\clU_{i}, \varphi_{i}) := (\clU_{x_{i}}, \varphi_{x_{i}})$ ($i = 1, \ldots, M$) satisfies 
\eqref{E:finite.numbr.coord.nbhds.cover.any.ball}. 

Let $L \in [1, \infty)$ be larger than the Lipschitz constants 
for $\varphi_{i}$ and $\varphi_{i}^{-1}$ ($i = 1, \ldots, M$). 
By \eqref{E:finite.numbr.coord.nbhds.cover.any.ball}, if $\eta$ is sufficiently small then for every $x \in \Rcl$ there exists a coordinate neighborhood $(\clU_{i}, \varphi_{i})$ 
(with $\varphi_{i}$ Lipschitz) s.t.\ $\mcl{B}_{\eta}(x) \subset \clU_{i}$. Then 
        \begin{equation}  \label{E:phi.B.eta.L}
		\varphi_{i} \bigl( \mcl{B}_{\eta}(x) \bigr) 
		      \subset B_{L \eta} \bigl( \varphi_{i}(x) \bigr) \subset \RR^{d},
  	 \end{equation}
where $B_{L \eta} \bigl( \varphi_{i}(x) \bigr)$ is the ball in $\RR^{d}$ with center at $\varphi_{i}(x)$ and radius $L \eta$ (see \eqref{E:Euc.ball.defn}). Hence, applying $\varphi_{i}^{-1}$ to to both sides of \eqref{E:phi.B.eta.L}, we have, by \eqref{E:Lip.magnification.of.Hm} and \eqref{E:when.Haus.meas.=.Leb.meas} (or theorem \ref{T:Hausdorff.measure.gives.right.value.on.manifs}),
   \[
      \Hm^{d} \bigl( \mcl{B}_{\eta}(x) \bigr) 
        \leq L^{d} \Hm^{d} \bigl[ B_{L \eta} \bigl( \varphi_{i}(x) \bigr) \bigr]
          = L^{d} \mcl{L}^{d} \bigl[ B_{L \eta} \bigl( \varphi_{i}(x) \bigr) \bigr]
            = L^{d} \alpha(d) \, (L ^{d} \eta^{d}),
   \]
where $\mcl{L}^{d}$ denotes $d$-dimensional Lebesgue measure 
and $\alpha(d)$ is the volume of the unit ball in $\RR^{d}$. Let $b :=  L^{2d} \alpha(d)$. Similarly,
	\[
		\varphi_{i}^{-1} \Bigl( B_{\eta/L} \bigl( \varphi_{i}(x) \bigr) \Bigr) 
		         \subset \mcl{B}_{\eta}(x) \subset \D. 
	\]
Applying $\varphi_{i}$ to both sides:
   \[
      L^{d} \Hm^{d} \bigl( \mcl{B}_{\eta}(x) \bigr) 
        \geq \Hm^{d} \bigl( \varphi_{i} \bigl[ B_{\eta}(x) \bigr] \bigr)  \geq
         \Hm^{d} \bigl[ B_{\eta/L} \bigl( \varphi_{i}(x) \bigr) \bigr] 
           = \mcl{L}^{d} \bigl[ B_{\eta/L} \bigl( \varphi_{i}(x) \bigr) \bigr]
             = \alpha(d) (L ^{-d} \eta^{d}).
   \]
Take $a := L^{-2d} \alpha(d)$. Thus, \eqref{E:vol.of.ball.approx.by.power} holds and the claim is proved.

Let $\delta \in (0,1)$. For $\delta > 0$, let $D(\delta, \Rcl)$ denote the $\delta$-packing number 
and $N(\delta, \Rcl)$ be the $\delta$-covering number of
$\Rcl$ (Pollard \cite[p.\ 10]{dP90.EmpPro}). (An alternative to using packing and covering numbers is to use Vitali's covering theorem (Giaquinta \emph{et al} \cite[Lemma 1, p.\ 30, Volume I]{mGgMjS98.cart.currents}, Simon \cite[Theorem 3.3, p.\ 11]{lS83.GMT}.)
Then by Pollard \cite[p.\ 10]{dP90.EmpPro},
	\begin{equation}  \label{E:N.and.D.relationships}
		N(\delta/2, \Rcl) \geq D(\delta, \Rcl) 
		        \geq N(\delta, \Rcl).
	\end{equation}
Since $\Rcl$ is compact $N(\delta/2, \Rcl)$ and, hence, $D(\delta, \Rcl)$ are finite. By definition of $D(\delta, \Rcl)$ there are 
$D(\delta, \Rcl)$ disjoint open balls with centers in $\Rcl$ and radius $\delta/2$. Obviously, these balls all lie in $\Rcl^{\delta}$. Similarly, by definition 
of $N(\delta, \Rcl)$ there are $N(\delta, \Rcl)$ closed balls in $\D$ and radius $\delta$ that cover $\Rcl$. Let $\epsilon \in (0, r)$. (Recall $r := \dim \Rcl$.) By definition of Hausdorff measure (appendix \ref{Chptr:Lip.Haus.meas.dim}), if $\delta \in (0,1)$ is sufficiently small that \eqref{E:vol.of.ball.approx.by.power} holds with $\eta \in (0, \delta]$, then by \eqref{E:vol.of.ball.approx.by.power} and \eqref{E:N.and.D.relationships}, 
   \begin{align*}
      \Hm^{d} (\Rcl^{\delta}) &\geq a D(\delta, \Rcl) (\delta/2) ^{d}  \\
         &\geq 2^{-d} a   \, \omega_{r-(\epsilon/2)}^{-1} \delta^{d-r+(\epsilon/2)}  \times
            \bigl[ \omega_{r-(\epsilon/2)} N(\delta, \Rcl) \delta^{r-(\epsilon/2)} \bigr]  \\
         &\geq 2^{-d} a   \, \omega_{r-(\epsilon/2)}^{-1} \delta^{d-r+(\epsilon/2)} \, 
               \Hm^{r-(\epsilon/2)}_{2 \delta}(\Rcl),
   \end{align*}
where $\omega_{r-(\epsilon/2)} > 0$ is the multiplicative constant in the definition 
of $\Hm^{r-(\epsilon/2)}_{2 \delta}$. (See \eqref{E:Hs.delta.defn}.) By assumption, $\omega_{r-(\epsilon/2)}$ is bounded for $\epsilon \in (0, r)$. Hence, $( 2^{-d} a   \, \omega_{r-(\epsilon/2)}^{-1} )^{-1}$ is uniformly bounded in $\epsilon \in (0,r)$.
Since $r - (\epsilon/2) < \dim \Rcl$, as $\delta \downarrow 0$, we have 
$\Hm^{r-(\epsilon/2)}_{2 \delta}(\Rcl) \to \Hm^{r-(\epsilon/2)} (\Rcl) = + \infty$. Therefore, eventually 
	\[
		\Hm^{r-(\epsilon/2)}_{2 \delta}(\Rcl) 
			\geq \Bigl( 2^{-d} a   \, \omega_{r-(\epsilon/2)}^{-1} \Bigr)^{-1}.
	\]  
Therefore, as $\delta \downarrow 0$, eventually, 
   \[
      \Hm^{d} (\Rcl^{\delta}) 
         \geq  \delta^{d-r + \epsilon/2} = \delta^{-\epsilon/2} \delta^{d-r + \epsilon}
           \geq \delta^{d-r + \epsilon},
   \]
since $0 < \delta < 1$. I.e., \eqref{E:Hm.Rdelta.geq.delta.to.codim} holds.

A similar argument proves \eqref{E:Hm.Rdelta.geq.delta.to.codim.times.Hm.R}: By \eqref{E:N.and.D.relationships},
   \begin{align*}
      \Hm^{d} (\Rcl^{\delta}) &\geq a D(\delta, \Rcl) (\delta/2) ^{d}  \\
         &\geq 2^{-d} a   \, \omega_{r}^{-1} \delta^{d-r} \times
               \bigl[ \omega_{r} N(\delta, \Rcl) \delta^{r} \bigr]  \\
         &\geq 2^{-d} a   \, \omega_{r}^{-1} \delta^{d-r} \Hm^{r}_{2 \delta}(\Rcl).
   \end{align*}
Since $\Hm^{r}(\Rcl)$ is finite, for $\delta > 0$ sufficiently small, 
$\Hm^{r}_{2 \delta}(\Rcl) \geq \tfrac{1}{2} \Hm^{r}(\Rcl)$.  This proves
\eqref{E:Hm.Rdelta.geq.delta.to.codim.times.Hm.R}.
  \end{proof}

 \emph{Confidence set for $\pi^{\Delta}$}: 
Suppose $X$ is drawn from an unknown probability distribution $P$. But it \emph{is} ``known'' that $P = P_{\theta_{true}}$, for some $\theta_{true} \in \Theta$, 
where $\{ P_{\theta}, \theta \in \Theta \}$ is a specific family of probability distributions on $\D$ and $\Theta$ some index set. The catch is $\theta_{true}$ is unknown. Suppose $P_{\theta}$ is absolutely continuous w.r.t.\ some measure $\mu$ on $\D$ and let $f_{\theta}$ be the density $d P_{\theta}/d \mu$. 
Let $\delta(X) :=  \text{dist} (X, \Ss)$, so $\delta(X)$ is random. 
Given $\Delta > 0$ fixed, let 
$\pi(\theta) := \pi^{\Delta} := \pi(\Delta, \theta) := P_{\theta} \bigl\{ \delta(X) \geq \Delta \bigr\}$, 
$\theta \in \Theta$. 
Let $x$ be the value of $X$ we observe. I.e., $x$ is the ``data''. $\theta_{true}$ is unknown but based on $x$ we can try to compute a confidence set for $\pi(\theta_{true})$. 

One way to do so is by inverting tests of the hypotheses $H_{0,p} : \pi(\theta) = p$ vs. 
$H_{1,p}: \pi(\theta) \neq p$ ($p \in [0,1]$). (So $H_{0}$ and $H_{1}$ do not represent homology groups.) 
Let $\Theta_{0,p} := \{ \theta \in \Theta : \pi(\theta) = p \}$. Thus, the null hypothesis is that $\theta_{true} \in \Theta_{0,p}$. The likelihood ratio test (LRT; Bickel and Doksum \cite[pp.\ 209--210]{pjBkaD77.MathStat}) is a recipe for creating such a test. For any $y \in \D$ let 
	\[
		\lambda(p, y) :=  \frac{ \sup \{ f_{\theta}(y) : 
		  \theta \in \Theta \} }{ \sup \{ f_{\theta}(y) : \pi(\theta) = p \} } .
	\]
Let $\alpha \in (0,1)$ be small and suppose we can calculate a number $c(\theta, p) := c(\theta, p; \alpha)$ that approximately satisfies
$\text{Prob}_{\theta} \, \bigl\{ \lambda(p,X) \geq c(\theta, p) \bigr\} = \alpha$ for every $\theta \in \Theta_{0}$ and $p \in (0,1)$. (Here again $X$ is random.)  
$\theta_{true}$ is unknown, but suppose we have an estimate, $\hat{\theta} = \hat{\theta}(x)$ of it for which theory tells us
$c(\hat{\theta}, p; \alpha)$ is approximately 
equal to $c(\theta_{true}, p; \alpha)$ for every $p$.  
Let  
$C(x) := \bigl\{ p \in [0,1] : \lambda(p,x)  < c(\hat{\theta}, p; \alpha) \bigr\}$
Then 
    \begin{multline*}
       \text{Prob}_{\theta_{true}} \bigl\{ \pi_{true} \in C(x) \} 
         = \text{Prob}_{\theta_{true}} \bigl\{ x \in \D : 
           \lambda(\pi_{true},x)  < c(\hat{\theta}, \pi_{true}; \alpha) \bigr\} \\
         \text{ is approximately equal to } \text{Prob}_{\theta_{true}} \bigl\{ x \in D : 
           \lambda(\pi_{true},x)  < c(\theta_{true}, \pi_{true}) \bigr\} \\
            \text{ is approximately equal to } 1 - \alpha .
    \end{multline*} 
Thus, $C(x)$ is an approximate $(1-\alpha)$-confidence set for $\pi(\Delta, \theta_{true})$ (Lehmann \cite[Theorem 4, p.\ 79]{elL93.StatHyps}). In practice it probably would be hard to compute such a thing.

   \begin{proof}[Proof of proposition \ref{P:deriv.blows.up}]
 First, take $\delta = \delta_{\langle \cdot, \cdot \rangle}$.  
Let $x \in \Ss$ be fixed but arbitrary and let $\varphi : \clU_{0} \to \RR^{d}$ be a coordinate neighborhood of $x$. Let $G_{0} := \varphi(\clU_{0})$ and let $\psi : G_{0} \to \clU_{0}$ be the inverse of $\varphi$. We may assume $\varphi(x) = 0$. Since $\Ss$ is locally compact, we may assume that $\Ss \cap \clU_{0}$ is relatively closed in $\clU_{0}$. (\emph{Pf:} $x$ has a relatively open neighborhood in $\Ss$ with compact closure \emph{in $\Ss$}. Thus, $x$ has a neighborhood $\mcl{A}$ in $\D$ s.t.\ the relative closure of $\mcl{A} \cap \Ss$ in $\Ss$ is compact. Thus, there exists $\mcl{C} \subset \D$ closed s.t.\ $\mcl{C} \cap \Ss$ is compact and $\mcl{A} \cap \Ss \subset \mcl{C} \cap \Ss$. Thus, $\mcl{K} := \mcl{C} \cap \Ss$ is closed. Replace $\clU_{0}$ by $\clU_{0} \cap \mcl{A}$. Thus, $\clU_{0} \subset \mcl{A}$ so we have $\clU_{0} \cap \Ss \subset \mcl{K}$. Therefore,
	\begin{equation*}
		\mcl{K} \cap \clU_{0} = (\mcl{C} \cap \Ss) \cap \clU_{0} 
		     = (\mcl{C} \cap \Ss) \cap (\Ss \cap \clU_{0}) 
		       = \mcl{K} \cap (\Ss \cap \clU_{0}) 
		        = \Ss \cap \clU_{0}.
	\end{equation*}
Thus, $\Ss \cap \clU_{0}$ is relatively closed in $\clU_{0}$ as desired.)  

By proposition \ref{P:geod.cnvx.nbhds.exist} there are neighborhoods $\clU_{1}$ and $\clU$ of $x$ s.t.\  $\clU \subset \overline{\clU_{1}} \subset \clU_{0}$, $\overline{\clU_{1}}$ is compact, and $\clU$ is geodesically convex.  
 		
Define $\Gamma_{\D,x'}^{d \times d}$ ($x' \in \clU$) as in lemma \ref{L:coord.maps.are.Lip}, with $\D$ in place of $M$. Let $\mu_{1}(x') \geq \ldots \geq \mu_{d}(x') > 0$ be the eigenvalues of $\Gamma_{\D,x'}$. By part \ref{I:bold.mu.bounds.eigvals} of lemma \ref{L:coord.maps.are.Lip}, there exists $\blds{\mu} \in (0, \infty)$ s.t.\
	\begin{equation}  \label{E:Eigvals.bdd.away.from.0.infty}
		\mu_{1}(x') \leq \blds{\mu}^{2} \text{ and } \, 1/ \mu_{d}(x') 
		  \leq \blds{\mu}^{2} 
		  \text{ for every } x' \in \clU.
	\end{equation}

Pick $\eta_{0} = \eta_{0}(x) > 0$ so small that $\mcl{B}_{\eta_{0}}(x)$, the open ball in $\D$ about $x$ with radius $\eta_{0}$, lies in $\clU$. In particular, 
$\overline{ \mcl{B}_{\eta_{0}}(x) }$ is compact. Let $G = \varphi(\clU)$ and 
$H = \varphi \bigl( \mcl{B}_{\eta_{0}}(x) \bigr) \subset G$. By making $\eta_{0}$ smaller if necessary, we may choose $r_{0} \in (0, \infty)$ so that 
	\[
		\overline{H} \subset B_{r_{0}}(0) \subset G. 
	\]
Here $B_{r_{0}}^{k}(0) := \bigl\{ y \in \RR^{d} : |y| < r_{0} \bigr\}$. (See \eqref{E:Euc.ball.defn}.) Recall that $\varphi(x) = 0$. 

By part \eqref{I:phi.psi.Lip} of lemma \ref{L:coord.maps.are.Lip}, we may assume there exists $\blds{\mu}$ s.t.\ \eqref{E:Eigvals.bdd.away.from.0.infty} holds and so does the following.
	\begin{multline}   \label{E:phi.and.psi.are.Lip}
	  \text{The map $\varphi$ is Lipschitz on $\mcl{B}_{\eta_{0}}(x)$, $\psi$ is Lipschitz 
	   on $\overline{H}$, and $\blds{\mu}$, as in \eqref{E:Eigvals.bdd.away.from.0.infty},} 
	   \\
	    \text{ is a Lipschitz constant for both. }
	\end{multline}

Let $\eta \in (0, \eta_{0})$. \emph{Claim:}
	\begin{equation} \label{E:B.eta/2.not.fill.B.eta}
		\text{There exists } r > 0 \text{ s.t.\ } B_{r}(0) \nsubseteq \varphi \bigl( \mcl{B}_{\eta/2}(x) \bigr)
			\text{ but } B_{r}(0) \subset \varphi \bigl( \mcl{B}_{\eta}(x) \bigr) \subset B_{r_{0}}(0).
	\end{equation}
Suppose not. Then
	\begin{equation}   \label{E:Br.in.phi.B.eta.means.Br.in.phi.B.eta/2}
		B_{r}(0) \subset \varphi \bigl( \mcl{B}_{\eta}(x) \bigr) \text{ implies }
			B_{r}(0) \subset \varphi(\mcl{B}_{\eta/2}(x)).
	\end{equation}
Let $r_{\eta} := \inf \Bigl\{ |y| : y \in \RR^{d} \setminus \varphi \bigl( \mcl{B}_{\eta}(x) \bigr)  \Bigr\} > 0$. Then $B_{r_{\eta}}(0) \subset \varphi \bigl( \mcl{B}_{\eta}(x) \bigr)$. Hence, \eqref{E:Br.in.phi.B.eta.means.Br.in.phi.B.eta/2} implies 
	\begin{equation}   \label{E:B.r.eta.in.phi.B.eta/2}
		B_{r_{\eta}}(0) \subset \varphi(\mcl{B}_{\eta/2}(x)).
	\end{equation}
For $n = 1, 2, \ldots$, pick 
	\[
		y_{n} \in \Bigl[ \varphi \bigl( \mcl{B}_{\eta}(x) \bigr) \Bigr]^{c}
	\] 
s.t.\ $|y_{n} | \downarrow r_{\eta}$. WLOG, $\{ y_{n} \}$ converges to some 
$y_{\infty} \in \Bigl[ \varphi \bigl( \mcl{B}_{\eta}(x) \bigr) \Bigr]^{c}$. 
Thus, $|y_{\infty}| = r_{\eta}$ so $y_{\infty} \in \overline{ B_{r_{\eta}}(0) }$. Therefore, by \eqref{E:B.r.eta.in.phi.B.eta/2},
	\begin{equation}   \label{E:y.infty.in.phi.B.eta/2}
		 y_{\infty} \in \overline{ \varphi \bigl( \mcl{B}_{\eta/2}(x) \bigr) }.
	\end{equation}
 Let $x_{\infty} := \psi(y_{\infty})$, so 
$x_{\infty} \in \clU \setminus \mcl{B}_{\eta}(x)$. Therefore, $\delta(x_{\infty}, x) \geq \eta$. But, \eqref{E:y.infty.in.phi.B.eta/2} implies $\delta(x_{\infty}, x) \leq \eta/2$. This contradiction proves the claim \ref{E:B.eta/2.not.fill.B.eta}.

I.e., we can choose $r \in (0, r_{0})$ large enough that $B_{r}(0)$ is a subset 
of $\varphi \bigl[ \mcl{B}_{\eta}(x) \bigr]$ but  is \emph{not} a subset 
of $\overline{ \varphi \bigl( \mcl{B}_{\eta/2}(x) \bigr) }$. Thus, $\psi \bigl[ B_{r}(0) \bigr]  \nsubseteq \mcl{B}_{\eta/2}(x)$. Hence, by \eqref{E:phi.and.psi.are.Lip}, 
	\begin{equation}   \label{E:eta/2mu.leq.r.leq.mu.eta}
	      B_{\eta/(2 \blds{\mu})}(0) \subset \varphi \bigl[ \mcl{B}_{\eta/2}(x) \bigr] \text{ so }
		\frac{\eta}{2 \blds{\mu}} \leq r; \quad
		 B_{r}(0) \subset \varphi(\mcl{B}_{\eta}(x)) \text{ so } r \leq \blds{\mu} \eta.
	\end{equation}

Since $x \in \Ss$ and $\F$ is complete w.r.t.\ $\rho$, as we assume throughout this section, there exists $\epsilon > 0$ (independent of $\eta$) s.t.\ there exist 
$x_{1}, x_{2} \in \D \setminus \Ss$ arbitrarily close to $x$ 
with 
	\begin{equation}   \label{E:rho.phi.x1.phi.x2.geq.eps}
		\rho \bigl[ \Phi(x_{1}), \Phi(x_{2}) \bigr] \geq \epsilon.
	\end{equation}  
Define $\mcl{B}_{\eta}'(x) := \mcl{B}_{\eta}(x) \setminus \Ss$. In particular, we may choose 
$x_{1}, x_{2} \in \psi \bigl[ B_{r}(0) \bigr] \cap \mcl{B}_{\eta/2}'(x)$, the open ball in $\D$ about $x$ with radius $\eta/2$ with that property. Let 
$y_{i} = \varphi(x_{i}) \in \Bigl(\varphi \bigl[ \mcl{B}_{\eta/2}(x) \bigr] 
\setminus \varphi( \Ss ) \Bigr) \cap B_{r}(0)$ ($i=1,2$). 
Thus, $y_{1}, y_{2} \notin \varphi \bigl[ \Ss \cap \mcl{B}_{\eta_{0}}(x) \bigr]$.

\emph{Claim:} For any $\beta > 0$ with $B_{\beta}(y_{2}) \subset B_{r}(0)$  we can find $y_{2}' \in B_{\beta}(y_{2})$ s.t.\ the line segment, $L$, joining $y_{1}$ and $y_{2}'$ does not intersect $\varphi(\Ss) \cap \clU$ and $\Phi_{\ast}$ is defined at almost all points of $\psi(L)$. To see this, let $b = |y_{2} - y_{1}|$ and 
$\partial B_{b}(y_{1}) := \bigl\{ y \in \RR^{d} : |y - y_{1}| = b \bigr\}$. I.e., $\partial B_{b}(y_{1})$ is the $(d-1)$-sphere centered at $y_{1}$ with radius $b$. Let
		\[
			F(s, z) = y_{1} + s (z - y_{1}) \in \overline{ B_{b}(y_{1}) }, 
				\quad s \in (0, 1 ], \; z \in \partial B_{b}(y_{1}).
		\]
Thus, $F$ is Lipschitz and
	\begin{equation*}
	   F^{-1}(w) 
	       = \Bigl( b^{-1} | w-y_{1} |,  \; b | w-y_{1} |^{-1} ( w-y_{1} ) + y_{1} \Bigr)
	       \in (0,1] \times \partial B_{b}(y_{1}), 
		\quad w \in B_{b}(y_{1}) \setminus \{ y_{1} \}.
	\end{equation*}   

Let $\Rcl := \bigl\{ x' \in \mcl{B}_{\eta_{0}}'(x) :  \Phi_{\ast, x} \text{ is not defined.}\bigr\}$. Then by \eqref{E:Phi.star.defined.a.e.}, we have $\Hm^{d} (\Rcl) = 0$. In particular, $\Rcl$ is $\Hm^{d}$-measurable (Federer \cite[p.\ 54]{hF69}). Therefore, by lemma \ref{L:loc.Lip.image.of.null.set.is.null} and \eqref{E:phi.and.psi.are.Lip}, 
$\mcl{L}^{d} ( R' ) = 0$, 
where $R' = \varphi(\Rcl) \cap B_{b}(y_{1}) \setminus \{ y_{1} \}$. But $F^{-1}$ is locally Lipschitz on $B_{b}(y_{1}) \setminus \{ y_{1} \}$. Therefore, by lemma \ref{L:loc.Lip.image.of.null.set.is.null} 
	\begin{equation} \label{E:0.=.H.d.F.invrs}
		0 = \Hm^{d} \bigl[ F^{-1}(R') \bigr].
	\end{equation}
For $z \in \partial B_{b}(y_{1})$, let
	\[
		S_{z} = \bigl\{ s \in (0,1] : (s,z) \in F^{-1}(R') \bigr\}.
	\]
Therefore, by \eqref{E:0.=.H.d.F.invrs} and Federer \cite[2.10.27, p.\ 190]{hF69} we have
	\begin{equation*}
		0 = \int_{\partial B_{b}(y_{1})}^{\ast} \int_{S_{z}} ds  \, \Hm^{d-1}(dz).
	\end{equation*}
Hence, by lemma \ref{L:0.star.int.means.0.a.e.} below, for $\Hm^{d-1}$-almost all 
$z \in \partial B_{b}(y_{1})$ we have 
$\Hm^{1} \bigl\{ s \in (0,b] : F(s,z) \in R' \bigr\} = \mcl{L}^{1}(S_{z})=0$. 

Now, $\dim \Ss < d-1$ by assumption. By lemma \ref{L:loc.Lip.image.of.null.set.is.null} in appendix \ref{Chptr:Lip.Haus.meas.dim} and \eqref{E:phi.and.psi.are.Lip}, $\dim \Bigl[ \varphi \bigl( \Ss \cap \mcl{B}_{\eta_{0}}(x) \bigr) \Bigr] < d-1$. In particular, 
	\begin{equation}   \label{E:H.d-1.phi.S.0}
		\Hm^{d-1} \Bigl[ \varphi \bigl( \Ss \cap \mcl{B}_{\eta_{0}}(x) \bigr) \Bigr] = 0.
	\end{equation}
Let $\pi_{2} : (0, 1 ] \times \partial B_{b}(y_{1}) \to \partial B_{b}(y_{1})$ be projection onto the second factor. Then, since $F^{-1}$ is locally Lipschitz on $B_{b}(y_{1}) \setminus \{ y_{1} \}$, by \eqref{E:comp.of.Lips.is.Lip}, we have that $\pi_{2} \circ F^{-1}$ is locally Lipschitz map of $B_{b}(y_{1}) \setminus \{ y_{1} \}$ into $\partial B_{b}(y_{1})$. Hence, by \eqref{E:H.d-1.phi.S.0} and lemma \ref{L:loc.Lip.image.of.null.set.is.null} again, we have that 
	\begin{equation*}
		\Hm^{d-1} \left( \pi_{2} \circ F^{-1} 
		    \Bigl[ \varphi \bigl( \Ss \cap \mcl{B}_{\eta_{0}}(x) \bigr) \Bigr] \right) = 0.
	\end{equation*}

Now,
$\pi_{2} \circ F^{-1}  \Bigl[ \varphi \bigl( \Ss \cap \mcl{B}_{\eta_{0}}(x) \bigr) \Bigr] \subset \partial B_{b}(y_{1})$. 
Since $\dim \partial B_{b}(y_{1}) = d-1$, it follows that we may pick $y_{2}' \in \partial B_{b}(y_{1}) \cap B_{\beta}(y_{2})$ s.t.\ the line segment, $L$, joining $y_{1}$ to $y_{2}'$ does not intersect $\varphi \bigl( \Ss \cap \mcl{B}_{\eta_{0}}(x) \bigr)$ 
and $\Phi_{\ast, x}$ is defined for almost all $x \in \psi(L)$. This completes the proof of the claim. 

By \eqref{E:rho.phi.x1.phi.x2.geq.eps}, since $\Phi$ is continuous off $\Ss$, for $\beta > 0$ sufficiently small, 
	\begin{equation}  \label{E:y2'.near.y2}
		\rho \Bigl( \Phi \bigl[ \psi(y_{1}) \bigr], \Phi \bigl[ \psi(y_{2}') \bigr] \Bigr) =
		     \rho \Bigl( \Phi (x_{1}), \Phi \bigl[ \psi(y_{2}') \bigr] \Bigr) \geq \epsilon/2 
			\text{ and } y_{2}' \in B_{\beta}(y_{2}) \cap B_{r}(0).
	\end{equation} 

Similarly, we can find $y_{3}' \in B_{r}(0) \setminus \overline{ \varphi \bigl( \mcl{B}_{\eta/2}(x) \bigr) }$ s.t.\ 
the line segment $\overline{y_{2}'y_{3}'}$ joining $y_{2}'$ and $y_{3}'$ does not intersect 
$\varphi \bigl( \Ss \cap \mcl{B}_{\eta_{0}}(x) \bigr)$, $\Phi_{\ast, \psi(y)}$ is defined for $\mcl{L}^{1}$-almost all $y$ in $\overline{y_{2}'y_{3}'}$, and $y_{3}' - y_{2}'$ and $y_{1} - y_{2}'$ are linearly independent. (Recall $y_{1} \neq y_{2}$.) Of course, 
$\overline{y_{2}'y_{3}'} \subset B_{r}(0) \subset \varphi \bigl[ \mcl{B}_{\eta_{0}}'(x)  \bigr]$. To see this, first observe that, by \eqref{E:B.eta/2.not.fill.B.eta}, there exists $y_{3} \in B_{r}(0) \setminus \varphi(\mcl{B}_{\eta/2}(x))$. We may assume $y_{3} - y_{2}'$ and $y_{1} - y_{2}'$ are linearly independent. Now proceed as before with $y_{2}'$ playing the role of $y_{1}$ and $y_{3}$ playing the role of $y_{2}$. Denote the point then corresponding to the new $y_{2}'$ by $y_{3}'$.

Define 
$\xi_{1} : \bigl[ 0, | y_{2}' - y_{1} | + | y_{3}' - y_{2}' | \bigr] 
\to B_{r}(0) \subset \varphi \bigl( \mcl{B}_{\eta}(x) \bigr)$ by
	\begin{equation}   \label{E:xi.curve.defn}
		\xi_{1}(s) = 
			\begin{cases}
				y_{1} + \frac{s}{| y_{2}' - y_{1} |} ( y_{2}' - y_{1} ),  
				            & \text{ if } 0 \leq s \leq | y_{2}' - y_{1} |, \\
				y_{2}' + \frac{s-| y_{2}' - y_{1} |}{| y_{3}' - y_{2}' |} ( y_{3}' - y_{2}' ),  
				   & \text{ if } | y_{2}' - y_{1} | < s \leq | y_{2}' - y_{1} | + | y_{3}' - y_{2}' |.
			\end{cases}
	\end{equation}
Thus, the image of $\xi_{1}$ consists of two line segments lying in $B_{r}(0)$, \emph{viz.}\ the one joining $y_{1}$ to $y_{2}'$ and the one joining $y_{2}'$ to $y_{3}'$. Note that $\xi_{1}$ does not intersect $\varphi(\Ss \cap \clU)$. Since $y_{3}'$ and $y_{3}' - y_{2}'$ and $y_{1} - y_{2}'$ are linearly independent, the function $\xi_{1}$ is one-to-one. Notice that $\xi_{1}$ is parametrized by arclength.
Since $y_{1}, y_{2}', y_{3}'  \in B_{r}(0)$, by \eqref{E:eta/2mu.leq.r.leq.mu.eta},
	\begin{equation}  \label{E:bound.on.lngth.of.xi1}
		\text{Length of the curve $\xi_{1}$ } \leq 4 r < 4 \blds{\mu} \eta.
	\end{equation}

Now let $\alpha : [0, \lambda ] \to \mcl\D$, where $\lambda := | y_{2}' - y_{1} | + | y_{3}' - y_{2}' |$, be 
the curve $\alpha(s) = \psi \circ \xi_{1}(s)$. Note that, since the image of $\xi_{1}$ lies in $B_{r}(0)$ we have, by \eqref{E:B.eta/2.not.fill.B.eta}, that
	\begin{equation} \label{E:alpha.image.in.B.eta}
		\text{Image of } \alpha \subset \mcl{B}_{\eta}'(x).
	\end{equation}

Since $\psi$ and $\xi$ are both one-to-one, so is $\alpha$. Then, by \eqref{E:phi.and.psi.are.Lip}, 
\eqref{E:bound.on.lngth.of.xi1}, and lemma \ref{L:loc.Lip,avg.bound.on.Haus.meas} (with $h := \psi$),
	\begin{equation}  \label{E:lngth.alph.leq.4.mu2.eta}
		\text{The length of $\alpha$ is no greater than $4 \blds{\mu}^{2} \eta$. }
	\end{equation}
However, by \eqref{E:y2'.near.y2}, as one moves along $\alpha$ from $\psi( y_{1} )$ to $\psi( y_{2}' )$ in $\D$ the point $\Phi \circ \alpha$ moves a distance of at least $\epsilon/2$. Thus, adapting formula \eqref{E:alpha.length.integral},
	\begin{equation}  \label{E:half.eps.length.int}
		\frac{\epsilon}{2} 
			\leq \int_{0}^{\lambda} 
				\bigl\| \Phi_{\ast} \circ \alpha'(s) \bigr\|_{\F, \Phi \circ \alpha(s)} \; ds,
	\end{equation}
where $\alpha'(s) = \alpha_{\ast} ( d/du)_{u=s}$. But by \eqref{E:operator.norm.of.Phi.ast}
	\begin{equation}   \label{E:Riemann.norms.operator.norm}
			\bigl\| \Phi_{\ast} \circ \alpha'(s) \bigr\|_{\F, \Phi \circ \alpha(s)}
				\leq \bigl\| \alpha'(s) \bigr\|_{\D, \alpha(s)} \, | \Phi_{\ast}  |_{\alpha(s)}.
	\end{equation}

By part \ref{I:alpha'.shorter.longer.mu} of lemma \ref{L:coord.maps.are.Lip}, 
	\begin{equation}  \label{E:lwr.uppr.bnd.on.alpha'}
	           \blds{\mu} \geq \bigl\| \alpha'(s) \bigr\|_{\D, \alpha(s)}
			  \geq \blds{\mu}^{-1}.
	\end{equation}

We have, 
	\begin{equation}  \label{E:avg.size.Phi.alng.curve.Riem}
		\frac{1}{\text{length of } \alpha} 
		     \int_{\alpha[0, \lambda]}  | \Phi_{\ast,x}  | \, \Hm^{1}(dx) 
				= \frac{1}{\text{length of } \alpha} \int_{0}^{\lambda} 
				  | \Phi_{\ast, \alpha(s)}| \, \| \alpha'(s) \|_{\D, \alpha(s)} \, ds,   
	\end{equation}
where equality follows from the change of variables formula (Federer \cite[Theorem 3.2.5, pp.\ 244 and 282]{hF69}) 
(See also Giaquinta \emph{et al} \cite[Theorem 2, p.\ 75, Volume I]{mGgMjS98.cart.currents}).
Therefore, by \eqref{E:avg.size.Phi.alng.curve}, \eqref{E:avg.size.Phi.alng.curve.Riem}, \eqref{E:Riemann.norms.operator.norm}, \eqref{E:half.eps.length.int}, and \eqref{E:lngth.alph.leq.4.mu2.eta}, the average size of the derivative of $\Phi$ along the curve 
$\alpha$ is
	\begin{align*}
		\frac{1}{\text{length of } \alpha} \int_{0}^{\lambda} 
		             | \Phi_{\ast, \alpha(s)} | \,  \bigl\| \alpha'(s) \bigr\|_{\D, \alpha(s)} \, ds 
		   &\geq \frac{1}{\text{length of } \alpha} \int_{0}^{\lambda} 
				\bigl\| \Phi_{\ast} \circ \alpha'(s) \bigr\|_{\F, \Phi \circ \alpha(s)} \; ds  \\
		   &\geq \frac{1}{4 \blds{\mu}^{2} \eta} \, \epsilon/2.
	  \end{align*}
This, and \eqref{E:alpha.image.in.B.eta}, prove part (\ref{I:avg.size.bggr.C/eta}) of the proposition (since $\epsilon$ is independent of $\eta$).

We assert that for some constant $K > 0$, independent of $\eta$ and $r$,
	\begin{equation} \label{E:r.avg.dist.from.0.along.xi1}
		\int_{0}^{\lambda} \bigl| \xi_{1}(s) \bigr| \, ds \geq K \lambda r.
	\end{equation}
(We may take $K := \min \left\{ \tfrac{1}{32 \blds{\mu}^{2}},  \tfrac{1}{256 \blds{\mu}^{4}} \right\}$.) Thus, the average distance from the origin to $\xi_{1}$ is at least $K r$. For proof of \eqref{E:r.avg.dist.from.0.along.xi1} see below.

Recall that for now $\delta :=\delta_{\langle \cdot, \cdot \rangle}$. Now, by \eqref{E:phi.and.psi.are.Lip}, 
	\begin{align}  \label{E:|xi1|.leq.mu.delta}
		\bigl| \xi_{1}(s) \bigr| &= \bigl| \xi_{1}(s) - 0 \bigr| \notag \\ 
			&= \bigl| \xi_{1}(s) - \varphi(x)\bigr| \\ 
			&= \bigl| \varphi \circ \alpha(s) - \varphi(x)\bigr|  \notag  \\
			&\leq \blds{\mu} \delta \bigl[ \alpha(s), x \bigr]. \notag 
	\end{align}
By \eqref{E:lwr.uppr.bnd.on.alpha'}, $\blds{\mu} \geq \| \alpha'(s) \| \geq \blds{\mu}^{-1}$. Therefore, analogously to \eqref{E:avg.size.Phi.alng.curve.Riem}, by \eqref{E:alpha.length.integral}, \eqref{E:lwr.uppr.bnd.on.alpha'}, \eqref{E:|xi1|.leq.mu.delta}, \eqref{E:r.avg.dist.from.0.along.xi1}, and \eqref{E:eta/2mu.leq.r.leq.mu.eta},
	\begin{align}   \label{E:avg.dist.to.x.bound}
		\text{average distance from } \alpha \text{ to } x
		&= \frac{ \int_{0}^{\lambda} \delta \bigl[ \alpha(s), x \bigr] \| \alpha'(s) \|_{\D, \alpha(s)} \, ds}
		        { \int_{0}^{\lambda} \| \alpha'(s) \|_{\D, \alpha(s)} \, ds} \notag \\
		&\geq  \frac{\blds{\mu}^{-1} \int_{0}^{\lambda} \delta \bigl[ \alpha(s), x \bigr]  
		  \, ds}
		            {\blds{\mu} \int_{0}^{\lambda} \, ds} \notag \\
		&\geq \frac{ \int_{0}^{\lambda} \bigl| \xi_{1}(s) \bigr|  \, ds}
		            {\blds{\mu}^{3} \lambda} \\
		&\geq \frac{ K \lambda r}
		            {\blds{\mu}^{3} \lambda} \notag \\
		&= \blds{\mu}^{-3} K r \notag \\
		& \geq \blds{\mu}^{-4} K \eta/2. \notag
	\end{align}
This proves part (\ref{I:avg.dist.bggr.C.eta}) for $\delta$.

Now let $\delta$ be any metric on $\D$ s.t.\ \eqref{E:two.deltas.locally.equivalent} holds. Let $\Hm^{1, 1}$ denote 1-dimensional Hausdorff measure computed relative to $\delta_{\langle \cdot, \cdot \rangle}$ and let $\Hm^{2, 1}$ denote 1-dimensional Hausdorff measure computed relative to $\delta$. For $\eta > 0$, let $\mcl{B}_{1,\eta}(x) \subset \D$ denote ball centered at $x$ with radius $\eta$ computed relative to $\delta_{\langle \cdot, \cdot \rangle}$ and let $\mcl{B}_{2,\eta}(x) \subset \D$ denote ball centered at $x$ with radius $\eta$ computed relative to $\delta$. 

Pick $\eta_{0} = \eta_{0}(x) > 0$ so small that $\mcl{B}_{1, K(\clU) \eta_{0}}(x) \subset \clU$. Then, by \eqref{E:two.deltas.locally.equivalent}, we have $\mcl{B}_{2, \eta_{0}}(x) \subset \clU$. Let $H_{1} = \varphi \bigl( \mcl{B}_{1, K(\clU) \eta_{0}}(x) \bigr) \subset G$. Let $H_{2} = \varphi \bigl( \mcl{B}_{2, \eta_{0}}(x) \bigr) \subset G$. By making $\eta_{0}$ smaller if necessary, we may choose $r_{0} \in (0, \infty)$ so that 
	\[
		\overline{H}_{2} \subset \overline{H}_{1} \subset B_{r_{0}}(0) \subset G. 
	\] 
Let $\eta \in (0, \eta_{0})$. Now go through the construction above for $\delta_{\langle \cdot, \cdot \rangle}$ with $\eta$ replaced by $\eta/ K(\clU)$. Then 
	\begin{equation*}
		\text{Image of } \alpha \subset \mcl{B}_{1, \eta/K(\clU)}'(x) 
		  \subset \mcl{B}_{2, \eta}'(x).
	\end{equation*}

By lemma \ref{L:metric.dominance.means.Haus.meas.dominance}, making $\clU$ smaller if necessary, there is a Borel measurable function $M : \D \to (0, \infty)$ s.t.\ $M$ and $1/M$ are both bounded on $\clU$ and, using an obvious notation,
	\begin{equation}  \label{E:M.Hm1.12}
		M(y)^{-1} \Hm^{2,1}(dy) \leq \Hm^{1,1}(dy) \leq M(y) \Hm^{2,1}(dy), \quad y \in \clU.
	\end{equation}
Let $b := := b(\clU) := \sup_{y \in \clU} M(y) < \infty$. Let $A := \alpha \bigl( [0,\lambda] \bigr) \subset \clU \subset \D$. Thus, by \eqref{E:M.Hm1.12},
	\begin{equation} \label{E:H21.A.H11.A}
		\Hm^{2,1}(A) = \int_{A} \Hm^{2,1}(dy) \leq \int_{A} M(y) \Hm^{1,1}(dy) \leq b \Hm^{1,1}(A).
	\end{equation}
Hence, by \eqref{E:M.Hm1.12}, \eqref{E:H21.A.H11.A}, and 
lemma \ref{L:meas.dominance.means.integral.dominance} below,
	\begin{multline*}
		\text{average size of the derivative of } \Phi \text{ along } \alpha 
		  \text{ w.r.t. } \delta \\
			\begin{aligned}
				=& \frac{ \int_{A}  | \Phi_{\ast,x}  | \, \Hm^{2,1}(dy ) }{\Hm^{2,1}(A) } \\
				\geq& \frac{ \int_{A}  | \Phi_{\ast,x}  | \, M(y)^{-1} \, \Hm^{1,1}(dy ) }
				  { b \Hm^{1,1}(A) } \\
				\geq& b^{-1} \frac{ \int_{A}  | \Phi_{\ast,x}  | \, \Hm^{1,1}(dy ) }
				  { b \Hm^{1,1}(A) } \\
				=& b^{-2} \text{average size of the derivative of } 
				  \Phi \text{ along } \alpha \text{ w.r.t. } 
				\delta_{\langle \cdot, \cdot \rangle} \\
				\geq& b^{-2} C(x)/ \bigl( \eta / K(\clU) \bigr),
			\end{aligned}
	\end{multline*}
since part \ref{I:avg.size.bggr.C/eta} of the proposition holds for $\delta_{\langle \cdot, \cdot \rangle}$ and we are applying it with $\eta / K(\clU)$ in place of $\eta$. Now, by Ash \cite[Theorem A5.15, p.\ 387]{rbA72}, we may assume the cover by sets $\clU$ is countable and locally finite: $\clU_{1}, \clU_{2}, \ldots$. Replace $C(x)$ by 
	\begin{equation*}
		\min \bigl\{ b(\clU_{i})^{-2} K(\clU_{i}) C(x) : i = 1, 2, \ldots s.t.\ x \in \clU_{i} \bigr\}.
	\end{equation*}
(Thus, in the preceding the minimization is over $i$, not $x$, which is fixed.) Therefore, part \ref{I:avg.size.bggr.C/eta} of the proposition holds for $\delta$.

Similarly, we have,
	\begin{multline*}
		\text{average distance from } \alpha \text{ to } x \text{ w.r.t. } \delta \\
			\begin{aligned}
				=& \frac{ \int_{A}  \delta(y, x) \, \Hm^{2,1}(dy ) }{\Hm^{2,1}(A) } \\
				\geq& \frac{ \int_{A}  K(\clU)^{-1} \delta_{\langle \cdot, 
				  \cdot \rangle}(y, x) M(y)^{-1} \, 
				        \Hm^{1,1}(dy ) }{ b \Hm^{1,1}(A) } \\
				\geq& b^{-1} K(\clU)^{-1} \frac{ \int_{A}  \delta_{\langle \cdot, 
				  \cdot \rangle}(y, x) \, 
				       \Hm^{1,1}(dy ) }{ b \Hm^{1,1}(A) } \\
				=& (b^{2}K(\clU))^{-1} \text{average distance from } 
				  \alpha \text{ to } x \text{ w.r.t. } 
				      \delta_{\langle \cdot, \cdot \rangle} \\
				\geq& (b^{2}K(\clU))^{-1} C(x) \bigl( \eta / K(\clU) \bigr),
			\end{aligned}
	\end{multline*}
since part \ref{I:avg.dist.bggr.C.eta} of the proposition holds for $\delta_{\langle \cdot, \cdot \rangle}$. Therefore, part \ref{I:avg.dist.bggr.C.eta} of the proposition holds for $\delta$.
   \end{proof}   
   
   \begin{lemma} \label{L:0.star.int.means.0.a.e.}
Let $\phi$ be a measure on a set $X$ and let $f : X \to [0, + \infty]$. If
	\begin{equation}  \label{E:f.star.int.=.0}
		\int^{\ast} f \, d \phi = 0
	\end{equation}
then $f = 0$ $\phi$-almost everywhere. (See Federer \cite[2.4.2, p.\ 81]{hF69} for definition.)
   \end{lemma}

  \begin{proof} 
Let $A := f^{-1}(0, +\infty] \subset X$. Suppose $\delta := \min \bigl\{\phi (A), 1 \bigr\} > 0$. For $n = 1, 2, \ldots$, 
let $A_{n} := f^{-1}[1/n, +\infty]$. Then
	\begin{equation*}
		A = \bigcup_{n} A_{n}.
	\end{equation*}
Hence,  
	\begin{equation*}
		0 < \phi(A) \leq \lim_{n \to \infty} \phi ( A_{n} ).
	\end{equation*}
Thus, for some $n$ we have $\phi ( A_{n} ) > \delta/2$. (Subadditivity is part 
of Federer's \cite[2.1.2, p.\ 53]{hF69} definition of measure. E.g., Hausdorff measures are subadditive. See \eqref{E:Hm.is.mono.and.count.subadditive}.)

If there are no upper functions (Federer \cite[2.4.2, p.\ 81]{hF69}) for $f$, then $\int^{\ast} f \, d \phi = \infty$, contradicting \eqref{E:f.star.int.=.0}. Let $u : X \to \RR$ be an upper function for $f$. Then $u(x) \geq 1/n$ for $\phi$-almost all $x \in A_{n}$. Hence, 
	\begin{equation*}
		\sum_{y \in \RR} y \cdot \phi \bigl[ u^{-1}(y) \bigr] \geq \frac{\delta}{2 n} > 0.
	\end{equation*}
But this also contradicts \eqref{E:f.star.int.=.0}. It follows that $f = 0$ $\phi$-almost everywhere, as desired.
 \end{proof}

\begin{lemma}  \label{L:meas.dominance.means.integral.dominance}
Let $(X, \X)$ be a measurable space, let $\mu_{1}$, $\mu_{2}$ be two measures on $(X, \X)$ and let $M : X \to [0, + \infty]$ be Borel measurable and satisfy
	\begin{equation}  \label{E:mu2.leq.M.mu1}
		\mu_{2}(A) \leq \int_{A} M(x) \, \mu_{1}(dx), \text{ for every } A \in \X.
	\end{equation}
Then if $f : X \to [0, + \infty]$ is Borel then
	\begin{equation*}  \label{E:mu2.int.leq.M.mu1.int}
		\int f(x) \, \mu_{2}(dx) \leq \int M(x) f(x) \, \mu_{1}(dx).
	\end{equation*}
\end{lemma}

\begin{proof}
Let $f : X \to [0, + \infty]$ be Borel.  Then by Ash \cite[Theorem 1.5.5(a), p.\ 38]{rbA72}, there exists a sequence, $\{ f_{n} \}$ of non-negative step functions increasing everywhere to $f$.  Then by \eqref{E:mu2.leq.M.mu1}, clearly
	\[
		\int f_{n}(x) \, \mu_{2}(dx) \leq \int M(x) f_{n}(x) \, \mu_{1}(dx)
	\]
Let $n \uparrow \infty$ and apply Monotone Convergence (Ash \cite[Theorem 1.6.2, p.\ 44]{rbA72}).
\end{proof}

\begin{proof}[Proof of \eqref{E:r.avg.dist.from.0.along.xi1}]
We use the following lemma, proved below.

  \begin{lemma}  \label{L:avg.dist.from.strght.line}
Let $N = 2, 3, \ldots$ and let $x, y \in \RR^{N}$ be distinct. Consider the straight line (parametrized by arclength) joining $x$ to $y$:
	\begin{equation*}
		\xi(s) := x + |y-x|^{-1} s (y-x), \quad 0 \leq s \leq |y-x|
	\end{equation*}
Then the average vector length along this line satisfies
	\begin{equation} \label{E:avg.dist.from.0.bnd}
		|y-x|^{-1} \int_{0}^{|y-x|} \bigl| \xi(s) \bigr| \, ds \geq \tfrac{1}{8} 
		  \max \bigl\{ |x|, |y| \bigr\}.
	\end{equation}
  \end{lemma}

Since $y_{3}' \notin \varphi \bigl( \mcl{B}_{\eta/2}(x) \bigr)$, by \eqref{E:eta/2mu.leq.r.leq.mu.eta}, 
	\begin{equation}  \label{E:y3.lngth.lwr.bnd}
		| y_{3}' | \geq \frac{r}{2 \blds{\mu}^{2} }.
	\end{equation}
Recall that $\lambda := | y_{2}' - y_{1} | + | y_{3}' - y_{2}' |$. 
Then, since $y_{1}, y_{2}', y_{3}' \in B_{r}(0)$,
	\begin{equation}  \label{E:lambda.leq.4r}
		\lambda = | y_{2}' - y_{1} | + | y_{3}' - y_{2}' | \leq 4 r.
	\end{equation}
Applying lemma \ref{L:avg.dist.from.strght.line} to each of the two segments 
  in $\xi_{1}$, we have
	\begin{equation}   \label{E:xi.lngth.smpl.lwr.bnd}
		8 \int_{0}^{\lambda} \bigl| \xi_{1}(s) \bigr| \, ds 
			\geq |y_{2}' - y_{1}| \max \bigl\{ |y_{1}|, |y_{2}'| \bigr\} 
			     + |y_{3}' - y_{2}'| \max \bigl\{ |y_{2}'|, |y_{3}'| \bigr\}.
	\end{equation}

Suppose $|y_{3}' - y_{2}'| < \tfrac{r}{4 \blds{\mu}^{2}}$. Then $|y_{2}'| \geq \tfrac{r}{4 \blds{\mu}^{2}}$ by \eqref{E:y3.lngth.lwr.bnd}. So by \eqref{E:lambda.leq.4r} and \eqref{E:xi.lngth.smpl.lwr.bnd},
	\begin{equation}  \label{E:4.int.bnd.w/.y2'.close.to.y3}
		8 \int_{0}^{\lambda} \bigl| \xi_{1}(s) \bigr| \, ds \geq |y_{2}' - y_{1}| |y_{2}'| 
		  + |y_{3}' - y_{2}'| |y_{2}'|  
			\geq \frac{ \lambda r}{4 \blds{\mu}^{2}}.
	\end{equation}
	
Now suppose $|y_{3}' - y_{2}'| \geq \tfrac{r}{4 \blds{\mu}^{2}}$. Then by \eqref{E:lambda.leq.4r} and \eqref{E:y3.lngth.lwr.bnd}, we have
	\[
		\frac{ |y_{3}' - y_{2}' |}{ \lambda } \geq \frac{1}{16 \blds{\mu}^{2}}.
	\]
Hence, by \eqref{E:xi.lngth.smpl.lwr.bnd},
	\begin{equation} \label{E:4.int.bnd.w/.y2'.far.from.y3}
		8 \int_{0}^{\lambda} \bigl| \xi_{1}(s) \bigr| \, ds \geq |y_{3}' - y_{2}'| |y_{3}'|  
			\geq \frac{ \lambda r }{32 \blds{\mu}^{4}}.
	\end{equation}

\eqref{E:r.avg.dist.from.0.along.xi1} follows from \eqref{E:4.int.bnd.w/.y2'.close.to.y3} and \eqref{E:4.int.bnd.w/.y2'.far.from.y3}.

\end{proof}

\begin{proof}[Proof of lemma \ref{L:avg.dist.from.strght.line}]
 Suppose $x \ne y$. WLOG $|y| \geq |x|$. Let $u = |y-x|^{-1} (y-x)$. So $u$ is a unit vector parallel to $\xi(s) $. There exist $v \in \RR^{N}$ perpendicular to $u$ and $a, b \in \RR$ s.t., $a u + v = x$ and $b u + v = y$. Then not both $a$ and $b$ are 0. Now $a^{2} + |v|^{2} = |x|^{2} \leq |y|^{2} = b^{2} + |v|^{2}$, so $|b| \geq |a|$. Moreover, by the (Cauchy-)Schwarz inequality (Stoll and Wong \cite[Theorem 3.1, p.\ 79]{rrSetW68.LinearAlgebra}), we have 
	\[
		b = u \cdot y =|y-x|^{-1} \bigl( |y|^{2} - x \cdot y \bigr) \geq 0.
	\]  
I.e., $b \geq |a|$. We also have 
	\[
		|y-x| = b-a.
	\]

We may reparametrize $\xi$  by defining $t = s+a$. Then
	\begin{equation*}
		\xi(t) :=  v + t u, \quad a \leq t \leq b
	\end{equation*} 
Let $c = |v| \geq 0$. Thus, we wish to bound $|y-x|^{-1} \int_{a}^{b} \sqrt{c^{2} + t^{2}} \, dt$. Note that
	\[
		\sqrt{c^{2} + t^{2}} \geq \tfrac{1}{\sqrt{2}} \bigl( c + |t| \bigr) 
		       \geq \tfrac{1}{2} \bigl( c + |t| \bigr), \quad t \in \RR.
	\] 
We thus have
	\begin{equation}  \label{E:int.of.sqrt.lwr.bnd}
		\int_{a}^{b} \sqrt{c^{2} + t^{2}} \, dt \geq \tfrac{1}{2} \int_{a}^{b} \bigl( c + |t| \bigr) \, dt.
	\end{equation}

First, suppose $0 \leq a \leq b$. Then 
	\begin{align}  \label{E:lwr.bnd.when.a.geq.0}
		\int_{a}^{b} \bigl( c + |t| \bigr) \, dt &= (b-a) \bigl[ c + \tfrac{1}{2} (b+a) \bigr] 
		  \notag \\
			 &= \tfrac{1}{2} |y-x| \bigl[ (c+a) + (c+b) \bigr] \notag \\
			 &= \tfrac{1}{2} |y-x| \Bigl[ \bigl( |v|+a |u| \bigr) + \bigl( |v|+b |u| \bigr) \Bigr] \\
			 &\geq \tfrac{1}{2} |y-x| \bigl( |x|+ |y| \bigr) \notag \\
			 &\geq |y-x| \cdot \tfrac{1}{2} \max \bigl\{ |x|, |y| \bigr\}. \notag 
	\end{align}
	
Next, suppose $a < 0 < b$. Then
	\begin{align}  \label{E:lwr.bnd.when.a.<.0}
		\int_{a}^{b} \bigl( c + |t| \bigr) \, dt 
		        &= \int_{a}^{0} ( c - t ) \, dt + \int_{0}^{b} ( c + t ) \, dt  \notag \\
		        &= -a \bigl( c - \tfrac{1}{2} a \bigr) + b \bigl( c + \tfrac{1}{2} b \bigr)  \\
		        &\geq -\tfrac{1}{2} a \bigl( c -  a \bigr) + \tfrac{1}{2} b \bigl( c + b \bigr) \notag. 
	\end{align}
Now, $|y-x| = b - a$. Therefore, since $b \geq |a| = -a$, we have $\tfrac{b}{b-a} \geq \tfrac{1}{2}$. Since, $a < 0$, we have $\tfrac{-a(c-a)}{(b-a)} \geq 0$. We also have $c + b = |v| + |b u| \geq |v + b u| = |y| = \max\{ |x|, |y| \}$. Therefore, from \eqref{E:lwr.bnd.when.a.<.0},
	\begin{align}  \label{E:lwr.bnd.whn.a.<.0.in.trms.x.y}
		\int_{a}^{b} \bigl( c + |t| \bigr) \, dt 
			&\geq \tfrac{1}{2} |y-x| \left[ \frac{-a}{b-a} (c-a) 
			        + \frac{b}{b-a} (c+b) \right] \notag \\
			&\geq \tfrac{1}{2} |y-x| \cdot \tfrac{1}{2} (c+b) \\
			&\geq \tfrac{1}{4} |y-x| \max\{ |x|, |y| \}. \notag 
	\end{align}
The lemma follows from \eqref{E:int.of.sqrt.lwr.bnd}, \eqref{E:lwr.bnd.when.a.geq.0}, and \eqref{E:lwr.bnd.whn.a.<.0.in.trms.x.y}.
\end{proof}

The following was adapted from \cite[Appendix A]{spE.fact.anal.long}.

   \begin{lemma} \label{L:Eigen.cont.}
      If $M$ is a symmetric $q \times  q$ (real) matrix ($q$, a given positive integer), let 
      $\Lambda (M) = \bigl( \lambda_{1}(M), \ldots, \lambda _{q}(M) \bigr)$, where 
      $\lambda _{1}(M) \geq \ldots \geq \lambda _{q}(M)$ are the eigenvalues of $M$. 
      Let $\| M \|$ be the Frobenius or Hilbert-Schmidt norm \eqref{E:matrix.norm}, 
      $\| M \| = \sqrt{trace \, M M^{T}}$. Then $\Lambda$ 
      is a continuous function (w.r.t. $\| \cdot \|$). Moreover, if $N$ and 
      $M_{1}, M_{2},\ldots$ are all symmetric $q \times  q$ (real) matrices s.t.\ 
      $M_{j} \rightarrow N$ (w.r.t.\ $\| \cdot \|$, i.e., entrywise) as 
      $j \rightarrow \infty $, let $Q_{j}^{q \times q}$ be a matrix whose rows comprise an orthonormal basis of $\mathbb{R}^{q}$ consisting of eigenvectors of $M_{j}$. Then there is a subsequence $j(n)$
      s.t.\ $Q_{j(n)}$ converges to a matrix whose rows comprise a basis 
      of $\mathbb{R}^{q}$ consisting
      of unit eigenvectors of $N$.
   \end{lemma}

 \begin{proof}
  Let $M_{1}, M_{2},\ldots $ all be symmetric $q \times  q$ matrices. Suppose $M_{j} \rightarrow N$ as $j \rightarrow \infty$. Then $N$ is symmetric.
Let   $\mu _{i} = \lambda _{i}(N)$ ($i = 1, \ldots, q$). In particular, $\mu _{1} \geq \ldots \geq \mu _{q}$. Since   $M_{j} \rightarrow N $, $\{ M_{j} \}$ is bounded. Hence, $\bigl\{ (\lambda _{1}(M_{j}), \ldots, \lambda _{q}(M_{j})) \bigr\}$ is bounded (Marcus and Minc \cite[1.3.1, pp.\ 140--141]{mMhM64.MatrixThyIneq}). For each $n$, let   $v_{n1}, \ldots, v_{nq} \in \RR^{q}$ be orthonormal eigenvectors of $M_{n}$ corresponding to $\lambda _{1}(M_{n}), \ldots, \lambda _{q}(M_{n})$, resp. (Stoll and Wong \cite[Theorem 4.1, p.\ 207]{rrSetW68.LinearAlgebra}). 
In particular, $v_{n1}, \ldots, v_{nq}$ span $\RR^{q}$. Let $S^{q -1}$ be the $(q-1)$-sphere
	\[
		S^{q -1} = \{ x \in \RR^{q} : |x| = 1 \}.
	\]
Then $v_{n1}, \ldots, v_{nq} \in S^{q-1}$. By compactness of $S^{q -1}$ we may choose a subsequence, $\{ j(n) \}$, s.t.\  $v_{j(n)i}$ converges to some $v_{i} \in \mathbb{R}^{q}$ 
and $\lambda _{i}(M_{j(n)})$ converges to some $\nu_{i} \in \mathbb{R}$ as $n \rightarrow \infty$ 
($i = 1, \ldots, q$). The vectors $v_{1}, \ldots, v_{q}$ are orthonormal and therefore span $\RR^{q}$. Moreover, we have $M_{j(n)} v_{j(n)i} \to N v_{i}$ and $M_{j(n)} v_{j(n)i} = \lambda _{i}(M_{j(n)}) v_{j(n)i} \to \nu_{i} v_{i}$. Thus, $v_{1}, \ldots, v_{q}$ are orthonormal eigenvectors of $N$ with eigenvalues $\nu _{1} \geq \ldots \geq \nu _{q}$. 
Hence, $\{ \nu _{1}, \ldots, \nu _{q} \} \subset \{ \mu _{1}, \ldots, \mu _{q} \}$. But this does not take into account the multiplicity of the eigenvalues. Note that $v_{1}, \ldots, v_{q}$ must span $\RR^{q}$.

Let $\gamma_{1}, \ldots, \gamma_{k}$ be the set of \emph{distinct} values in $\{ \mu_{1}, \ldots, \mu_{q} \}$. Then $k \leq q$. By Stoll and Wong \cite[Theorem 4.1, p.\ 207]{rrSetW68.LinearAlgebra}, $\RR^{q} = S_{1} \oplus \cdots \oplus S_{k}$, where $S_{1}, \ldots, S_{k}$ are mutually orthogonal, $N x = \gamma_{j} x$ if $x \in S_{j}$, and $\dim S_{j}$ is the number of indices $i$ s.t.\ $\mu_{i} = \gamma_{j}$ ($j=1, \ldots, k$). Let $j = 1, \ldots, q$. If $\nu_{i} = \gamma_{j}$, then $v_{i} \in S_{j}$. But $v_{1}, \ldots, v_{q}$ are orthonormal. Hence, the number of $i$'s for which $\nu_{i} = \gamma_{j}$ is no larger than $\dim S_{j}$. But there are $q$ $\nu_{i}$'s and $\dim S_{1} + \cdots + \dim S_{k} = q$. Therefore, the number of $i$'s for which $\nu_{i} = \gamma_{j}$ is just $\dim S_{j}$. But as we just saw $\dim S_{j}$ is the number of indices $i$ s.t.\ $\mu_{i} = \gamma_{j}$. I.e., for every $j = 1, \ldots, k$, the number of $i$'s for which $\nu_{i} = \gamma_{j}$ is the same as the number of indices $i$ s.t.\ $\mu_{i} = \gamma_{j}$. So $\nu _{1}, \ldots, \nu _{q}$ are the same as $\mu _{1}, \ldots, \mu _{q}$ even taking multiplicity into account. 
  
 The preceding argument obviously goes through if, instead of $M_{1}, M_{2}, \ldots$, we had started with a  \emph{subsequence}
 of $M_{1}, M_{2}, \ldots$. Thus, any subsequence has a further subsequence 
 s.t.\ $\Lambda (M_{n})$ converges to $\Lambda (N)$ along that subsequence of a subsequence. 
  It follows that $\Lambda (M_{n}) \to \Lambda (N)$. I.e., $\Lambda$ is continuous.
  \end{proof}

  \begin{proof}[Proof of \eqref{E:Riem.metric.on.TD}] By Boothby \cite[Corollary (2.5), p.\ 183]{wmB75}, it suffices to show $inc_{\ast}$ is an immersion. 

Let $(\clU, \varphi)$ be a coordinate neighborhood of $\D$, 
let $V \subset \RR^{d} = \varphi(\clU)$, let $\psi = \varphi^{-1}$, and let 
$E_{i, \psi(z)} := \psi_{\ast} \left( \tfrac{\partial}{\partial y_{i}} \right) \restriction_{y=z}$ 
($i = 1, \ldots, d; z \in V$) be the coordinate frame on $(\clU, \varphi)$. 
Thus, for every $x \in \clU$, $E_{1,x}, \ldots, E_{d,x}$ are linearly independent.

Identify $\psi$ with $inc \, \circ \, \psi$. Then $\psi$ is $C^{\infty}$ and we may write $\psi = (\psi_{1}, \ldots, \psi_{k})$. Let $i = 1, \ldots, d$ and identify $E_{i}$ 
with $inc_{\ast}(E_{i})$. Then, regarding $\RR^{k}$ as its own coordinate neighborhood, by Boothby \cite[Theorem (1.6), p.\ 183]{wmB75}, if $y \in \clU$, 
    \begin{equation*}
      E_{i,y} = \sum_{j} \left( \frac{\partial \psi_{j}(z)}{\partial z_{i}} \right) \restriction_{z=\varphi(y)} \;
        \frac{\partial}{\partial \alpha_{j}} \restriction_{\alpha = y} .
    \end{equation*} 
As usual, we identify $E_{i,y}$ with 
$\bigl(\ldots, \bigl( \partial \psi_{j}(z)/\partial z_{i} \bigr) \restriction_{z=\varphi(y)}, \ldots  \bigr) \in \RR^{k}$. It is also convenient to denote this vector by $E_{i,z}$. The context should make it clear which usage we are employing. 

Let $\tilde{\clU} = \pi_{C}^{-1} (\clU)$ be the restriction of $T \D \restriction_{\Pf}$ 
to $\clU$. (See \eqref{E:pi.is.proj}.) Define 
$\tilde{\varphi} : (x, a_{1} E_{1,x} + \ldots + a_{d} E_{d,x}) \mapsto \bigl( \varphi(x), a_{1}, \ldots, a_{d} \bigr) \in V \times \RR^{d}$. Then $(\tilde{\clU}, \tilde{\varphi})$ is a generic coordinate neighborhood on $T \D$ (Boothby \cite[p.\ 331]{wmB75}). 
Let $\tilde{\psi} : V \times \RR^{d} \to \tilde{\clU}$ be the inverse of $\tilde{\varphi}$. 

Then, regarding $\RR^{2k}$ as its own coordinate neighborhood, by Boothby \cite[Definitions (4.3), p.\ 70 and (4.1), p.\ 69]{wmB75}) it suffices to show that $inc_{\ast} \circ \tilde{\psi}$ has full rank $2d$. As usual, let $D \psi(z)$ be the Jacobian matrix of $\psi$. It is just the $k \times d$ matrix whose $j^{th}$ column is $E_{j,z}$, regarding the latter as a column vector. 
By Boothby \cite[Corollaries (1.4), p.\ 108 and (1.7), p.\ 110]{wmB75}, $D \psi(z)$ is of full rank $d$. 

Identify $\tilde{\psi}$ and $inc_{\ast} \circ \tilde{\psi}$. Let $D \tilde{\psi}(z,a)^{2k \times 2d}$ be the Jacobian matrix of $\tilde{\psi}$. We then must show that $D \tilde{\psi}(z,a)$ has full rank $2d$. But it is easy to see that 
    \begin{equation*}
      D \tilde{\psi}(z,a) =
          \begin{pmatrix}
             D \psi(z)                              & 0^{k \times d} \\
             a D^{2} \psi(z)^{k \times d} & D \psi(z)
          \end{pmatrix} ,
    \end{equation*}
where $a D^{2} \psi(z)$ is a matrix involving $a$ and second derivatives of $\psi$. It does not matter what $a D^{2} \psi(z)$ is because it is immediate that $D \tilde{\psi}(z,a)$ is of full rank because 
$D \psi(z)$ is.
  \end{proof}

  \begin{proof}[Proof of lemma \ref{L:xi+.generates.manif.topol}] Let $\clU \subset \D$ be open and let $x \in \clU$. By the Normal Neighborhood theorem (Boothby \cite[Theorem (6.6), p.\ 335]{wmB75}), $x$ has a normal neighborhood 
$\mcl{N} \subset \D$. $\mcl{N} = Exp_{x}(W)$, where $W$ is a star-shaped neighborhood of 0 in $T_{x} \D$. Thus, $Exp_{x}^{-1}(\mcl{N} \cap \clU)$ is open in $W$. (By Boothby \cite[Theorem (7.1), p.\ 338]{wmB75}, $Exp_{x}$ is continuous.) The Riemannian metric 
on $\D$ induces a norm on $T_{x} \D$. Since the Riemannian metric on $\D$  is the restriction of the Euclidean metric 
on $\RR^{k}$ by \eqref{E:D.imbedded.in.Rk}, identifying $T_{x} \D$ in the obvious way, the norm on $W$ is just the Euclidean norm. (This is spelled out below.) Denote the norm 
by $| \cdot |$. Pick $\delta > 0$ so small that the open ball, $B_{\delta}(0)$ with radius $\delta$ about 0 in $W$ lies in $Exp_{x}^{-1}(\mcl{N} \cap \clU)$. 
Thus, $Exp \bigl[ B_{\delta}(0) \bigr| \subset \clU$. 
But from Boothby \cite[p.\ 333]{wmB75}  we know that if $X \in T_{x} \D$ 
then $\xi (x, Exp_{x} X) = |X|$. 
Hence, $Exp_{x} \bigl[ B_{\delta}(0) \bigr|$ is just a $\xi$-ball in $\D$. This, shows that 
the $\xi$-topology is finer than the usual manifold topology on $\D$. Now, $Exp_{x}$ is a diffeomorphism on $W$. This means that $Exp_{x} \bigl[ B_{\delta}(0) \bigr|$ is open in the manifold topology so the manifold topology is finer than the $\xi$-topology. q.e.d.

Now let $(\clU, \varphi)$ be a coordinate neighborhood of $\D$, 
let $V \subset \RR^{d} = \varphi(\clU)$, let $\psi = \varphi^{-1}$, and let 
$E_{i, \psi(z)} := \psi_{\ast} \left( \tfrac{\partial}{\partial y_{i}} \right) \restriction_{y=z}$ 
($i = 1, \ldots, d; z \in V$) be the coordinate frame on $(\clU, \varphi)$. 
Thus, for every $x \in \clU$, $E_{1,x}, \ldots, E_{d,x}$ are linearly independent.

Let $\tilde{\clU} = \pi_{C}^{-1} (\clU)$ be the restriction of $T \D \restriction_{\Pf}$ 
to $\clU$. (See \eqref{E:pi.is.proj}.) Define 
$\tilde{\varphi} : (x, a_{1} E_{1,x} + \ldots + a_{d} E_{d,x}) \mapsto \bigl( \varphi(x), a_{1}, \ldots, a_{d} \bigr) \in V \times \RR^{d}$. Let $\tilde{\psi} : V \times \RR^{d} \to \tilde{\clU}$ be the inverse of $\tilde{\varphi}$. By Boothby \cite[Lemma (6.1), p.\ 332]{wmB75} it suffices to show that with the usual topology on $V \times \RR^{d} \subset \RR^{2d}$ and the topology induced on $\tilde{\clU}$ by $\xi_{+}$ the set $\tilde{\clU}$ is open 
and the map $\tilde{\varphi}$ is a homeomorphism. 

Let $inc : \D \to \RR^{k}$ be inclusion. By corollary \ref{C:cont.diff.=.loc.Lip}, $inc$ is locally Lipschitz, in particular continuous, w.r.t.\ $\xi$ and $| \cdot |$, the Euclidean norm 
on $\RR^{k}$. But, by \eqref{E:D.imbedded.in.Rk}, $inc$ is an imbedding. Therefore, the restriction of $| \cdot |$ to $\D$ generates the topology on $\D$. (This means $\clU$ is open in the $\xi$ topology, but we already know that.)

Let $x \in \clU$ and $a_{1}, \ldots, a_{d} \in \RR$. 
Write $a := (a_{1}, \ldots, a_{d})^{1 \times d}$, where we use superscripts to indicate matrix dimension. Thus, $a$ is a row vector. Let $w := \varphi(x)$ 
and let $X := \sum_{i=1}^{d} a_{i} E_{i,x} \in T_{x} \D$. Thus, $(x, X) = \tilde{\psi}(w,a)$. Following our usual convention we regard $X \in \RR^{k}$.

Let $\mcl{B}_{\delta}(x)$ be the $\xi$-ball about $x \in \D$ with radius $\delta$. 
Pick $\delta > 0$ so small that $\mcl{B}_{\delta}(x) \subset \clU$. 
Let $y \in \mcl{B}_{\delta}(x)$. Then for \emph{any} $Y \in T_{y} \D$ 
we have $(y,Y) \in \tilde{\clU}$. In particular, if $\xi_{+} \bigl[ (y,Y), (x, X) \bigr] < \delta$ 
then $(y,Y) \in \tilde{\clU}$. This proves that $\tilde{\clU}$ is open in the $\xi_{+}$-topology on $T \D$.
 
For $\delta > 0$, let $\tilde{\mcl{B}}_{\delta}(x, X)$ 
be the $\xi_{+}$-ball about $(x, X)$ with radius $\delta$. (See \eqref{E:ball.defn}.) 
Let $\epsilon > 0$. We first show that by making $\delta$ sufficiently small, 
$\tilde{\varphi}( \tilde{\mcl{B}}_{\delta} ) \subset B_{\epsilon}^{2d}(w)$, the $\epsilon$-ball 
in $V \times \RR^{d} \subset \RR^{2d}$.

Identify $\psi$ with $inc \, \circ \, \psi$. Then we may write 
$\psi = (\psi_{1}, \ldots, \psi_{k})$. 
Since $\RR^{k}$ is its own coordinate neighborhood. It follows that $\psi$ is $C^{\infty}$. 
Let $i = 1, \ldots, d$ and identify $E_{i}$ with $inc_{\ast}(E_{i})$. Then, if $y \in \clU$, 
    \begin{equation*}
      E_{i,y} = \sum_{j} \left( \frac{\partial \psi_{j}(z)}{\partial z_{i}} \right) 
        \restriction_{z=\varphi(y)} \;
          \frac{\partial}{\partial \alpha_{j}} \restriction_{\alpha = y} .
    \end{equation*} 
As usual, we identify $E_{i,y}$ with 
$\bigl(\ldots, \bigl( \partial \psi_{j}(z)/\partial z_{i} \bigr) \restriction_{z=\varphi(y)}, \ldots  \bigr) \in \RR^{k}$. Thus,  \linebreak
$X = \left( \ldots,  \sum_{i=1}^{d} a_{i} \tfrac{\partial \psi_{j}(z)}{\partial z_{i}} 
\restriction_{z=\varphi(x)}, \ldots \right)$. Since $\psi$ is $C^{\infty}$ it follows that $E_{i,y}$ is $C^{\infty}$ in $y$. 

Let $(z,b) \in V \times \RR^{d}$ with $y := \psi(z) \in \clU$ and 
$b = (b_{1}, \ldots, b_{d})^{1 \times d}$. 
Let $Y := \sum_{i=1}^{d} b_{i} E_{i,y} \in T_{x} \D \subset \RR^{k}$. 
Since $\xi$ generates the topology on $\clU$ and $\varphi$ is continuous, 
there exists $\delta_{1} > 0$ s.t.\    
    \begin{equation} \label{E:xi(y,x)<delta.implies.|z-w|<eps/sqrt.2}
        \xi(y, x) < \delta_{1} \text{ implies } |z - w| < \epsilon/\sqrt{2}
    \end{equation}

Let $E_{x}^{d \times k}$ be the matrix whose $i^{th}$ row is $E_{x,i}$. Define $E_{y}$ similarly. 
Thus, the entries of $E_{y}$ are continuous in $y \in \clU$. Let $\eta > 0$ and suppose,
    \begin{align} \label{E:|Y-X|.sqrd.ineq}
      2\eta^{2} > |Y - X|^{2} 
        &= \left| \sum_{i=1}^{d} b_{i} (E_{i,y} - E_{i,x}) 
          + \sum_{i=1}^{d} (b_{i} - a_{i}) E_{i,x} \right|^{2} \notag \\
        &= \bigl| b (E_{y} - E_{x}) + (b-a) E_{x} \bigr|^{2} \notag \\
        &= b (E_{y} - E_{x}) (E_{y} - E_{x})^{T} b^{T} 
          + 2 b (E_{y} - E_{x}) E_{x}^{T} (b-a)^{T} \\
        &\qquad + (b-a) E_{x} E_{x}^{T} (b-a)^{T}  \notag \\
        &= b (E_{y} - E_{x}) \Bigl[ (E_{y} - E_{x})^{T} b^{T} 
          + 2 E_{x}^{T} (b-a)^{T} \Bigr] \notag \\
        &\qquad + (b-a) E_{x} E_{x}^{T} (b-a)^{T} \notag .
    \end{align}

Since the entries of $E_{y}$ are continuous in $y \in \clU$, $E_{y}$ is continuous in $y$ 
w.r.t.\ the Frobenius or Hilbert-Schmidt norm \eqref{E:matrix.norm} defined 
by $\| M \| = \sqrt{trace \, M M^{T}}$, where $M$ is a matrix. Therefore, by lemma \ref{L:Eigen.cont.}, this means the eigenvalues of $E_{y} E_{y}^{T}$ are continuous in $y$. 
Since $E_{1}, \ldots E_{d}$ are linearly independent in $\clU$, the smallest eigenvalue, 
$\lambda_{d}$, of $E_{x} E_{x}^{T}$ is strictly positive. Therefore, there exists 
$\delta_{2} \in (0, \delta_{1})$ s.t.\ if $\xi(y,x) < \delta_{2}$, then the smallest eigenvalue of $E_{y} E_{y}^{T}$ is at least $\lambda_{d}/2 > 0$. By \eqref{E:|Y-X|.sqrd.ineq}, we have
    \begin{equation*}
      \sqrt{\lambda_{d}/2} \, |b| \leq  \sqrt{ b E_{y} E_{y}^{T} b^{T} } = |b E_{y}| 
        = |Y| \leq |Y-X| + |X| < \sqrt{2} \eta + |X| .
    \end{equation*}
It follows that $b$ is bounded. As $y \to x$, $E_{y} \to E_{x}$. Therefore, there exists $\delta_{3} \in (0,\delta_{2})$ s.t.\ if $\xi(y,x) < \delta_{3} = \delta_{3}(\eta)$, then, 
    \begin{equation*}
      b (E_{y} - E_{x}) \Bigl[ (E_{y} - E_{x})^{T} b^{T} 
       + 2 E_{x}^{T} (b-a)^{T} \Bigr] < \eta^{2} .
    \end{equation*}
Hence, if $\xi(y,x) < \delta_{3}$, in particular, if $\xi_{+} \bigl[ (y,Y), (x,X) \bigr] < \delta_{3}$, then, by \eqref{E:|Y-X|.sqrd.ineq}, 
    \begin{equation*} 
      \eta^{2} > (b-a) E_{x} E_{x}^{T} (b-a)^{T} \geq \lambda_{d} |b-a|^{2} .
    \end{equation*}
Let $\eta := \epsilon \sqrt{ \lambda_{d}/2 }$, then we have 
$\xi_{+} \bigl[ (y,Y), (x,X) \bigr] < \delta_{3}$ implies $|b-a|^{2} < \epsilon^{2}/2$. 
Combining this with \eqref{E:xi(y,x)<delta.implies.|z-w|<eps/sqrt.2}, we get that 
$\xi_{+} \bigl[ (y,Y), (x,X) \bigr] < \delta_{3}$ implies $|z-w|^{2} + |b-a|^{2} < \epsilon^{2}$. I.e.,
$\bigl| (z,b) - (w,a) \bigr| < \epsilon$.
This proves that $\tilde{\varphi}$ is continuous on $\tilde{\clU}$ w.r.t.\ $\xi_{+}$ and 
$| \cdot |$.

For the other direction, again let $(w, a) \in V \times \RR^{d}$ and let $x = \psi(w)$. 
Thus, $\tilde{\psi}(w,a) = \bigl( \psi(w), a E_{x} \bigr) = (x, a E_{x})$. 
Let $(z, b)$ also be in $V \times \RR^{d}$. By corollary \ref{C:cont.diff.=.loc.Lip} again, 
for $z$ sufficiently close to $w$ there exists 
$K < \infty$, $\xi \bigl[ \psi(z), \psi(w) \bigr] < K |z-w|$. 
Therefore, $\psi$ is contiuous w.r.t.\ $| \cdot |$ and $\xi$. We have 
    \begin{equation*}
      | b E_{y} - a E_{x} | \leq \bigl| (b-a) E_{x}| + |b||E_{y} - E_{x}|.
    \end{equation*}
As $b \to a$, $\bigl| (b-a) E_{x}| \to 0$. As $z \to w$ we have 
$y = \psi(z) \to \psi(w) = x$ 
so $E_{y} \to E_{x}$. In summary, 
$\xi_{+} \bigl[ \tilde{\psi}(z,b), \tilde{\psi}(w,a) \bigr] \to 0$ 
as $(z,b) \to (w,a)$. This proves that $\tilde{\varphi}^{-1}$ is continuous 
w.r.t.\ $| \cdot |$ and $\xi_{+}$. The claim follows: $\xi_{+}$, and thus $\omega_{\D}$, generate the standard topology on $T \D$.  
  \end{proof}

  \begin{proof}[Proof of lemma \ref{L:Exp.loc.Lip}]
Let $x \in \D$. Let $\mcl{V}$ be a coordinate neighborhood of $x$ with coordinate map $\varphi : \mcl{V} \to \RR^{d}$. Let $H := \varphi(\mcl{V})$ so $H$ is an open subset of $\RR^{d}$. Let $\psi : H \to \mcl{V}$ be the inverse 
of $\varphi$. Let $\clU \subset \D$ be a neighborhood of $x$ satisfying $\overline{\clU} \subset \mcl{V}$ and $\overline{\clU}$ is compact. Let $G := \varphi(\clU)$, so $G$ is a relatively compact subset of $\RR^{d}$. 

Let $T \D \restriction_{\mcl{V}} := \{ X_{x} \in T \D : x \in \mcl{V} \}$. Here, the notation $X_{x}$ means $X_{x} \in T_{x} \D$. Define $\tilde{\varphi} : T \D \restriction_{\mcl{V}} \to H \times \RR^{d}$ as follows. If $X_{x} \in T \D \restriction_{\mcl{V}}$ then we can write 
$X_{x} = \sum_{i=1}^{d} y_{i} \, \psi_{\ast} \bigl( \partial/\partial w_{i} \restriction_{w=\varphi(x)} \bigr)$, 
where $w_{1}, \ldots, w_{d}$ are the coordinates in $\RR^{d}$. 
Define $\tilde{\varphi}(X_{x}) := \bigl( \varphi(x), (y_{1}, \ldots, y_{d}) \bigr) \in H \times \RR^{d}$. 
Then $\tilde{\varphi} : T \D \restriction_{\mcl{V}} \to H \times \RR^{d}$ is a coordinate map for $T \D$ 
(Boothby \cite[Lemma (6.1), p.\ 332]{wmB75}). In particlar, $\tilde{\varphi}$ is a bijection. 
Let $\tilde{\psi} : H \times \RR^{d} \to T \D \restriction_{\mcl{V}}$ be the inverse of $\tilde{\varphi}$. Thus,
$\tilde{\psi}\bigl( w, (y_{1}, \ldots, y_{d}) \bigr) = \sum_{i=1}^{d} y_{i} \, \psi_{\ast} (\partial/\partial u_{i} \restriction_{u = w} )
\in T_{\psi(w)} \D$, where $w \in H$ and $u = (u_{1}, \ldots, u_{d})$.

By assumption, $\D$ is an imbedded submanifold of $\RR^{k}$. Let $F = (F_{1}, \ldots, F_{k}): \D \to \RR^{k}$ be the imbedding. Then $F(\mcl{V})$ is a coordinate neighborhood in $F(\D)$ with coordinate map $\varphi \circ F^{-1}$. 
Now, $\RR^{k}$ with the identity $id : \RR^{k} \to \RR^{k}$ is a single coordinate neighborhood for $\RR^{k}$ and the tangent vectors on $\RR^{k}$ have the form $\sum_{i =1}^{k} a_{i} \, (\partial/\partial z_{i}) \restriction_{z = v}$, where $z_{1}, \ldots, z_{k}$ are the coordinates in $\RR^{k}$ and $v \in \RR^{k}$. Therefore, a single coordinate neighborhood of $T \RR^{k}$ is $T \RR^{k}$ itself with coordinate map
$\zeta : \sum_{j =1}^{k} a_{j} \, (\partial/\partial z_{j}) \restriction_{z = x} \mapsto \bigl( x, a_{1}, \ldots a_{k} \bigr) \in \RR^{2k}$. 
Let $\Psi := (\Psi_{1}, \ldots, \Psi_{k}) := F \circ \psi : H \to \RR^{k}$. Let $\partial \Psi/\partial u_{i} \restriction_{u=w} \in \RR^{k}$ be the vector, i.e., $k$-tuple, whose $j^{th}$ coordinate is $\partial \Psi_{j}/\partial u_{i} \restriction_{u=w}$ and let $D \Psi(w)$ be the $k \times d$ matrix whose element in the $i^{th}$ row and $j^{th}$ column is $\partial \Psi_{i}/\partial u_{j} \restriction_{u=w}$ ($i=1, \ldots, k$; $j=1, \ldots, d$). Then, regarding elements of $\RR^{d}$ as row vectors, by Boothby \cite[Theorem (1.6), p.\ 109]{wmB75}, we have 
	\begin{multline*}
	\Psi_{\ast} \left( \sum_{i=1}^{d} \alpha^{i} \frac{\partial}{\partial u_{i}} \vert _{u = w} \right)
	  = \sum_{j=1}^{k} \left( \sum_{i=1}^{d} \alpha^{i} \, \frac{\partial \Psi_{j}}{\partial u_{i}} \restriction_{u = w} \right)
	    \frac{\partial}{\partial z_{j}} \restriction_{z = \Psi(w)}, \\
	           w \in H, \alpha := (\alpha^{1}, \ldots, \alpha^{d}) \in \RR^{d}.
	\end{multline*}
(Here, we use ``$\alpha := (\alpha^{1}, \ldots, \alpha^{d})$'' to align with the notation in Boothby \cite[Theorem (1.6), p.\ 109]{wmB75}. 
Recall that $\Psi$ is already expressed in local coordinates.) Therefore,
    \begin{equation} \label{E:zeta.Fast.psi.formla}
      \zeta \circ F_{\ast} \circ \tilde{\psi}(w, \alpha) 
        = \bigl( \Psi(w), \alpha \, D \Psi(w)^{T} \bigr) \in \RR^{2k}.
    \end{equation}
I.e., the $(k+j)^{th}$ entry in the RHS of the preceding is $\sum_{i=1}^{d} \alpha^{i} \tfrac{\partial \Psi_{j}}{\partial u_{i}} \restriction_{u=w}$ ($j=1, \ldots, k$). 

Write $\Omega := \zeta \circ F_{\ast} \circ \tilde{\psi} : H \times \RR^{d} \to \RR^{2k}$. Then $\Omega$ is injective. 
We \emph{claim} that for every $r > 0$ $\bigl( \Omega \restriction_{G \times B_{r}^{d}(0)} \bigr)^{-1}$ (see  \eqref{E:Euc.ball.defn}) is Lipschitz w.r.t.\ $\xi_{+}$ and the Euclidean norm on $G \times B_{r}^{d}(0)$. Along the way we will show that for every $r > 0$ we have that $\Omega \restriction_{G \times B_{r}^{d}(0)}$ is Lipschitz w.r.t.\ the Euclidean norm on $G \times B_{r}^{d}(0)$ and $\xi_{+}$ (see \eqref{E:xi+.from.2.metrics}).

Let $r > 0$. Let $w,w' \in G$ and $y, y' \in B_{r}^{d}(0)$. So $\Psi(w), \Psi(w') \in F(\D)$. Write $\Delta := D \Psi^{T}$. We may think of $\xi$ as a metric on $F(\D)$. (In the notation of lemma \ref{L:imbedding.is.loc.Lip}, we may use the metric $(F^{-1})^{\ast} \xi$ on $F(\D)$.) Thus, $\xi \bigl[ \Psi(w), \Psi(w') \bigr] = \xi \bigl[ \psi(w), \psi(w') \bigr]$. Then, by \eqref{E:xi+.from.2.metrics} and \eqref{E:zeta.Fast.psi.formla}, 
	\begin{equation}   \label{E:omega.D.Omega}
		\xi_{+} \Bigl( \Omega \bigl[ w, y \bigr], \Omega \bigl[ w', y' \bigr]  \Bigr)^{2} 
			 = \xi \bigl[ \Psi(w), \Psi(w') \bigr]^{2} + \bigl| y \Delta(w) - y' \Delta(w') \bigr|^{2}.
	\end{equation}

$\Psi$ is smooth, injective, and has a smooth inverse. Therefore, by compactness of $\overline{G} \subset H$, there are constants $a, A$ s.t.\ 
    \begin{equation} \label{E:a.Delta.A.ineqs}
      0 < a < \bigl| y \Delta(w) \bigr| \leq \bigl\| \Delta(w) \bigr\| < A < \infty, 
        \text{ for every } w \in G \text{ and unit vector } y \in \RR^{d}. 
    \end{equation}
Here $\| \cdot \|$ is the operator norm 
(Rudin \cite[p.\ 185]{wR64.PMA}). Similarly, by corollary \ref{C:cont.diff.=.loc.Lip} in appendix \ref{Chptr:Lip.Haus.meas.dim} and the fact that $\Psi$ is $C^{\infty}$, there is $K < \infty$ s.t.\ $\bigl\| \Delta(w) - \Delta(w') \| < K |w-w'|$ for every $w, w' \in G$. By corollary \ref{C:cont.diff.=.loc.Lip} again, 
$\psi$ is Lipschitz on $G$ w.r.t.\ the Euclidean norm and $\xi$. Thus, increasing $K < \infty$ if necessary, for $w, w' \in G$ 
and $y, y' \in B_{r}^{d}(0)$ we have, by \eqref{E:a.Delta.A.ineqs}, 
   \begin{subequations} \label{E:w.and.y.parts}
	\begin{gather*}  
		\xi \bigl[ \Psi(w), \Psi(w') \bigr] = \xi \bigl[ \psi(w), \psi(w') \bigr] \leq K |w - w'| 
		   \text{ and } \tag{\ref{E:w.and.y.parts}a} \\
			\begin{split}
				\bigl| y \Delta(w) - y' \Delta(w') \bigr| &\leq 
				       \bigl| y \Delta(w) - y' \Delta(w) \bigr| + \Bigl| y' \bigl[ \Delta(w) 
				       - \Delta(w') \bigr] \Bigr| \\
				 &\leq |y-y'| \bigl\| \Delta(w) \bigr\| + |y'| \bigl\| \Delta(w) 
				   - \Delta(w') \bigr\| \\
				 &\leq A |y-y'| + r K |w - w'|.
			\end{split} \tag{\ref{E:w.and.y.parts}b}
	\end{gather*}
   \end{subequations}

Substituting \eqref{E:w.and.y.parts} into \eqref{E:omega.D.Omega} and applying the triangle inequality proves that $\Omega$ is Lipschitz on $G \times B_{r}^{d}(0)$.
	
Now consider $\Omega^{-1}$. Let $w, w' \in G$ and $y, y' \in B_{r}^{d}(0)$. By corollary \ref{C:cont.diff.=.loc.Lip} yet again, we have that $\varphi : \clU \to G$ is Lipschitz w.r.t.\ $\xi$ and Euclidean distance. Therefore, there exists $L < \infty$ s.t.\ 
	\begin{equation}  \label{E:varphi.is.Lip.on.U}
		| w - w' | \leq L \xi \bigl[ \psi(w), \psi(w') \bigr] = L \xi \bigl[ \Psi(w), \Psi(w') \bigr].
	\end{equation}

Let $\eta := \tfrac{a}{2Kr}$ and suppose first that $|w-w'| \geq \eta |y-y'|$. Then, by \eqref{E:varphi.is.Lip.on.U} and \eqref{E:omega.D.Omega},   
	\begin{multline} \label{E:|w-w'|.leq.eta |y-y'|}
	\bigl|(w,y) - (w', y') \bigr|^{2} = |w-w'|^{2} + |y-y'|^{2} \leq ( 1+ \eta^{-2}) |w-w'|^{2} \\
	  \leq ( 1+ \eta^{-2}) L^{2} \xi \bigl[ \Psi(w), \Psi(w') \bigr]^{2} 
	    \leq ( 1+ \eta^{-2}) L^{2} 
	      \xi_{+} \Bigl( \Omega \bigl[ w, y \bigr], \Omega \bigl[ w', y' \bigr]  \Bigr)^{2}. 
	\end{multline}

Now assume $|w-w'| < \eta |y-y'|$. Then, by \eqref{E:a.Delta.A.ineqs}, 
	\begin{align*}  \label{E:dominated.by.y.Delta}
		\bigl| y' \Delta(w') - y \Delta(w) \bigr| &\geq \bigl| (y'-y) \Delta(w) \bigr| -
		       \bigl| y' \Delta(w) - y' \Delta(w') \bigr| \notag \\
		       &\geq \bigl| (y'-y) \Delta(w) \bigr| -
		       |y'| \bigl\| \Delta(w) - \Delta(w') \bigr\| \notag \\
		  &\geq a |y'-y| - K |y'| |w - w' | \\
		  &\geq a |y'-y| - K |y'| \frac{a}{2Kr} |y' - y| \notag \\
		  &\geq a |y'-y| - (a/2) |y' - y| = (a/2) |y' - y|. \notag 
	\end{align*}
Combining this with \eqref{E:varphi.is.Lip.on.U} and applying \eqref{E:omega.D.Omega} yields
    \begin{multline*}
        \bigl|(w,y) - (w', y') \bigr|^{2} = |w - w'|^{2} + |y - y'|^{2} \\
          \leq L \xi \bigl[ \Psi(w), \Psi(w') \bigr]^{2} + (2/a) 
            \bigl| y' \Delta(w') - y \Delta(w) \bigr|^{2} 
              \leq (L + 2 a^{-1}) \xi_{+} \bigl[ \Omega(w,y), \Omega(w',y') \bigr]^{2}.
    \end{multline*}
Combining this with \eqref{E:|w-w'|.leq.eta |y-y'|} proves that $\Omega^{-1}$ is Lipschitz, as desired. This proves the claim that $\Omega$ and its inverse are both Lipschitz.
	
We have been regarding $Exp$ as being defined on $F(\D) \subset \RR^{2k}$. Let $\widetilde{Exp}$ be the exponential map defined intrinsically on $\D$, apart from the imbedding $F$. Thus, $Exp = \widetilde{Exp} \circ F_{\ast}^{-1} \circ \zeta^{-1}$. Now, as observed just before the statement of the lemma, we have that $\widetilde{Exp}$ is defined on all of the tangent bundle $T \D$. Therefore, by Boothby \cite[Theorem (7.1), p.\ 338]{wmB75}, $\widetilde{Exp}$ is differentiable on $T \D$. 
This means $\widetilde{Exp} \circ \tilde{\psi}$ is differentiable on $H \times \RR^{d}$. Therefore, by corollary \ref{C:cont.diff.=.loc.Lip} again, $\widetilde{Exp} \circ \tilde{\psi}$ is Lipschitz 
on $G \times B_{r}^{d}(0)$ w.r.t.\ the Euclidean norm and $\xi$. Thus, by \eqref{E:comp.of.Lips.is.Lip} and the fact, just proved, that $\Omega^{-1}$ is Lipschitz on $\Omega \bigl[ G \times B_{r}^{d}(0) \bigr]$, we have that 
$\widetilde{Exp} \circ \tilde{\psi} \circ \Omega^{-1}$ is Lipschitz on $\Omega \bigl[ G \times B_{r}^{d}(0) \bigr]$. But
	\begin{equation*}
	  \widetilde{Exp} \circ \tilde{\psi} \circ \Omega^{-1} 
	    = \widetilde{Exp} \circ \tilde{\psi} \circ \tilde{\psi}^{-1} 
	          \circ F_{\ast}^{-1} \circ \zeta^{-1}
	      = \widetilde{Exp} \circ F_{\ast}^{-1} \circ \zeta^{-1} = Exp.
	\end{equation*}
Thus, $Exp : F(\D) \times \RR^{k}$ is locally Lipschitz, as desired. (Of course, usually we identify $\D$ and $F(\D)$.)  \end{proof}
  
\begin{proof}[Proof of proposition \ref{P:tubular.nbhd.thm}]  
Let $x \in \Pf \subset \D$. Let $(\X, \varphi)$ be a coordinate neighborhood of $x$ in $\D$. So $\varphi : \X \to \RR^{d}$. We may assume $\varphi(x) = 0$. 
Let $\psi := \varphi^{-1} : H \to \X$, where $H := \varphi(\X) \subset \RR^{d}$. 
Let $e_{1}, \ldots, e_{d}$ be the usual unit coordinate vectors in $\RR^{d}$. 
Define $\mcl{W} := \X \cap \Pf$. We may assume that the closure, $\overline{\mcl{W}}$, \emph{in $\Pf$} is compact. ($\Pf$ itself may not be closed in $\D$.) Since $\Pf$ is an imbedded submanifold of $\D$, by Boothby \cite[Theorem (5.5), p.\ 78]{wmB75}, we have that $\Pf$ is a ``regular submanifold'' of $\D$. So if $R := \RR^{p} \cap H$, 
 where $\RR^{p} \subset \RR^{d}$ is the span of $e_{1}, \ldots, e_{p}$ and $p = \dim \Pf$, then we may assume that 
 $\varphi( \mcl{W} ) = R$. 
 
 Let $E_{iy} :=  := \psi_{\ast}(\partial/\partial z_{i})_{y}$ ($i = 1, \ldots, d$) be the coordinate frame at $y \in \mcl{W}$. The vector vectors $X_{iy} := E_{iy}$ ($i = 1, \ldots, p$) at $y \in \mcl{W}$ are linearly independent (in particular, nonzero) and tangent to $\Pf$ at $y$. In fact, $X_{1y}, \ldots, X_{py}$ span $T_{y} \Pf$. For $i = p+1, \ldots, d$ and $y \in \mcl{W}$, let $X_{iy} := E_{iy} - \Pi_{y} E_{iy} \neq 0$, where $\Pi_{y}$ is orthogonal projection onto $T_{y} \Pf$ (w.r.t.\ the Riemannian metric on $\D$). Specifically, we can write $X_{iy} = E_{iy} - \sum_{\ell=1}^{p} a_{i\ell}(y) X_{\ell y}$. The functions $a_{1\ell}(y), \ldots, a_{p\ell}(y)$ are smooth functions of $y$ because they are rational functions of the matrix, $\bigl( \langle E_{iy}, E_{jy} \rangle \bigr)$, of the Riemannian metric, which is of full rank and smooth in $y$.

Thus, for $y \in \mcl{W}$ we have that $X_{(p+1)y}, \ldots, X_{dy}$ span $(T_{y} \Pf)^{\perp}$. For $y \in \mcl{W}$, we have that $X_{iy} \in T_{y} \D$ ($i=1, \ldots, d$) are linearly independent because a vanishing nontrivial linear combination of $X_{1y}, \ldots, X_{dy}$ would translate into a vanishing nontrivial linear combination of $E_{1y}, \ldots, E_{dy}$. 

Define $\beta : R \times \RR^{d-p} \to N$ by 
	\begin{equation*}
		\beta(z,s) = \beta \bigl( z, (s_{p+1}, \ldots, s_{d}) \bigr) 
		          :=  \left( \psi(z), \sum_{i=p+1}^{d} s_{i} X_{i, \psi(z)} \right) \in N,
	\end{equation*}
where $z \in R \subset \RR^{p} \subset \RR^{d}$ and $s = (s_{p+1}, \ldots, s_{d}) \in \RR^{d-p}$. Note that, regarded as a map into $T \D$, $\beta$ is smooth (Boothby \cite[p.\ 331]{wmB75}).

Now consider another coordinate neighborhood, $\mcl{Y}$, of $x$ in $\D$ leading to a parametrization 
$\theta : S \times \RR^{d-p} \to N$, where $S$ is a neighborhood of 0 in $\RR^{p}$. Let $Y_{(p+1)y}, \ldots, Y_{dy}$ be the basis of $(T_{y} \Pf)^{\perp}$ for $y \in \mcl{Y}$ analogous to $X_{(p+1)y}, \ldots, X_{dy}$. 
If $y \in \X \cap \mcl{Y} \cap \Pf$ there exist unique $z \in R$, $w \in S$ that correspond to $y$. 
In fact, if $s \in \RR^{d-p}$ then there exists unique $t \in \RR^{d-p}$ s.t.\ $\beta(z, s) = \theta(w, t)$. The map $s \mapsto t$ is smooth. By Boothby \cite[Theorem (1.6), p.\ 109]{wmB75}, translating between $X_{(p+1)y}, \ldots, X_{dy}$ and $Y_{(p+1)y}, \ldots, Y_{dy}$ involves linear combinations of derivatives of smooth coordinate maps and parametrizations. We conclude that the coordinate neighborhoods $R \times \RR^{d-p}$and $S \times \RR^{d-p}$ with the parametrizations $\beta$ and $\theta$ are ``compatible'' (Boothby \cite[Definition (1.1), p.\ 52]{wmB75}). Therefore, by Boothby \cite[Theorem (1.3a), p.\ 54]{wmB75}, $N$ is a $d$-dimensional differentiable manifold. 

In fact, $N$ is an immersed submanifold of $T \D$. This is easily seen by considering the map 
$\tilde{\varphi} \circ \beta : R \times \RR^{d-p} \to H \times \RR^{d}$, where $\tilde{\varphi}: T \X \to H \times \RR^{d}$ is the coordinate map on $T \X$ corresponding to $\varphi$ (Boothby \cite[p.\ 331]{wmB75}).
 
Note that, by \eqref{E:D.is.cmpct.Riem.manif} and Hopf-Rinow (Boothby \cite[Theorem (7.7), p.\ 343]{wmB75}), $Exp$ is defined on the entire tangent bundle $T \D$. \emph{Claim:} $Exp \circ \beta : R \times \RR^{d-p} \to \D$ has full rank 
at $\varphi(x)=0 \in \RR^{d}$. Since $Exp \circ \beta (z, 0) = \psi(z)$ ($z \in R$, so $Exp \circ \beta(0) = x$) we have 
for $i = 1, \ldots, p$,  
$(Exp \circ \beta)_{\ast} (\partial/\partial z_{i} \restriction_{z=0, s=0} ) = \psi_{\ast}(\partial/\partial z_{i} \restriction_{z=0} ) = X_{ix}$. 
(Recall that $x = \psi(0)$.) 

Let $i = p+1, \ldots, d$ and define $\lambda(t) := (x, t X_{ix})$ ($t \in \RR$). 
Then $\beta_{\ast} (\partial/\partial s_{i} \restriction_{z=0,s=0} ) = \lambda_{\ast}(d/dt \restriction_{t=0} )$. 
Therefore, by Boothby \cite[Lemma (6.4), p.\ 334]{wmB75}, 
	\begin{multline*}
		(Exp \circ \beta)_{\ast} \left( \frac{\partial}{\partial s_{i}} 
		  \restriction_{z=0, s=0} \right) 
		  = Exp_{\ast} \circ \beta_{\ast} \left( \frac{\partial}{\partial s_{i}} 
		    \restriction_{z=0, s=0} \right) \\
		    = Exp_{\ast} \circ \lambda_{\ast} \left(\frac{d}{dt} \restriction_{t=0} \right)
		      = (Exp \circ \lambda)_{\ast} \left(\frac{d}{dt} \restriction_{t=0} \right) 
		        = X_{ix}. 
	\end{multline*}
Since $X_{ix}$ ($i=1, \ldots, d$) are linearly independent the claim, that $Exp \circ \beta$ is of full rank, is proved.

Let $F := \varphi \circ Exp \circ \beta$. Then $F(z,0) = (z,0)$ for $z \in R$ and $F$ is a $C^{\infty}$ map from a neighborhood of $0$ in $\RR^{d}$ to a neighborhood of $0$ in $\RR^{d}$ and the Jacobian matrix $DF(0)$ is of full rank. Therefore, by the Inverse Function Theorem (Boothby \cite[Theorem (6.4), p.\ 42]{wmB75}) we have that $F$ is a diffeomorphism of some neighborhood of $0$ onto an open neighborhood of $0$. Therefore, WLOG we may assume that for some bounded neighborhood 
$\mcl{V}_{x} \subset \overline{\mcl{V}_{x}} \subset \mcl{W} \subset \Pf$ and for some $\epsilon'_{x} > 0$ we have that 
	\begin{multline} \label{E:Exp.is.diffeo.on.Vx.eps'}
		Exp \text{ is a diffeomorphism on }
		  \hat{N}^{\epsilon'_{x}} \restriction_{\mcl{V}_{x}} 
		    := \bigl\{ (x,v) \in N(\Pf, \D) : x \in \mcl{V}_{x}, \; |v| < \epsilon'_{x} \bigr\} \\
		      \text{ onto a set open in } \D \text{ that contains } \mcl{V}_{x}.
	\end{multline}
By proposition \ref{P:geod.cnvx.nbhds.exist}\eqref{I:normal.nbhd}, by making $\epsilon'_{x}$ smaller if necessary, we may assume  
	\begin{multline}  \label{E:Exp.is.diffeomorph.near.P}
		Exp \text{ is a diffeomorphism one each ball } 
		  T_{x} \D(\epsilon'_{x'}) := \bigl\{ (x, v) \in T_{x} \D : 
		    |v| < \epsilon'_{x} \bigr\}, \\
		    \xi(x', x) < \epsilon'_{x'}, \; x' \in \Pf.
	\end{multline} 
(So $(y,v) \in T_{x} \D(\epsilon'_{x})$ does not imply $y \in \Pf$, but if $x \in \Pf$, 
then $\hat{N}^{\epsilon'_{x}} \restriction_{\mcl{V}_{x}} \subset B_{x}(\epsilon'_{x})$)

By Lindel\"of's theorem (Simmons \cite[Theorem A, p.\ 100]{gfS63}) there exist 
$x_{1}, x_{2}, \ldots \in \Pf$ s.t.\ $\Pf \subset \bigcup_{i} \mcl{V}_{x_{i}}$.
By Boothby \cite[Lemma (4.1), p.\ 191]{wmB75}, there exists a locally finite refinement, 
$\clU_{1}, \clU_{2}, \ldots$, of $\{ \mcl{V}_{x_{i}}, \,i = 1, 2, \ldots \}$ that still covers $\Pf$ and for each $j$ there exists $i$ s.t.\ $\overline{\clU_{j}} \subset \mcl{V}_{x_{i}}$. 
 Since $\overline{\mcl{V}_{x}}$ is compact (since $\mcl{W}$ is) for every $x \in \Pf$, we have that $\overline{\clU_{i}}$ is compact for every $i$. Moreover, for each $i = 1,2 \ldots$, we may pick $\delta^{0}_{i} \in (0,\epsilon'_{x_{i}}]$ s.t.\ 
 $Exp$ is a diffeomorphism on $\hat{N}^{\delta^{0}_{i}} \restriction_{\clU_{i}}$

Generically, use $\blds{\delta}$ to stand for a sequence $\delta_{1}, \delta_{2}, \ldots$ with $\delta_{i} > 0$ for every $i = 1,2 \ldots$. Partial order such sequences component-wise. Next, we \emph{claim} 
	\begin{equation}  \label{E:Exp.1-1.on.union}
		\text{For some } \blds{\delta} \text{ we have that } 
		   Exp \text{ is one-to-one on } A(\blds{\delta}) := \bigcup_{i} \hat{N}^{\delta_{i}} \restriction_{\clU_{i}} \subset N.
	\end{equation} 
For any choice of $\blds{\delta}$ let $A_{0}(\blds{\delta}) := \varnothing$. For $i = 1, 2, \ldots$, define 
$M_{k} := M_{k}(\blds{\delta}) := N \cap \bigl[ \clU_{k} \times B_{\delta_{k}}^{d}(0) \bigr] = \hat{N}^{\delta_{k}} \restriction_{\clU_{k}}$ (see \eqref{E:ball.defn}) and
$ A_{i} :=  A_{i}(\blds{\delta}) := \bigcup_{k=1}^{i} M_{k}$. Thus, $A_{i}(\blds{\delta}) \uparrow A(\blds{\delta})$. 
Let $\mcl{G}_{i} := \bigcup_{k=1}^{i} \clU_{k}$. 

\eqref{E:Exp.1-1.on.union} is equivalent to 
	\begin{equation}  \label{E:Exp.1-1.on.finite.union}
		Exp \text{ is one-to-one on } A_{i}(\blds{\delta}) \text{ for every } i = 1, 2, \ldots.
	\end{equation}
A naive way to proceed would be, for each $i = 1, 2, \ldots$, to find $\blds{\delta}^{i}$ 
s.t.\ $Exp$ is one-to-one on $A_{i}(\blds{\delta})$ with $\blds{\delta} = \blds{\delta}^{i}$. That does not assure the existence of a single $\blds{\delta}$ s.t.\ \eqref{E:Exp.1-1.on.finite.union} holds. For example, it might turn out that 
$\liminf_{i \to 0} \blds{\delta}^{i}$, component-wise, might have some 0 components.

So instead we proceed in a more careful fashion. Consider this statement for a given 
$\ell = 0, 1, 2, \ldots$ 
and $\blds{\delta}^{\ell}$, 
	\begin{multline} \label{E:Exp.1-1.on.partial.union}
		Exp \text{ is one-to-one on } A_{\ell} ( \blds{\delta}^{\ell} ) \\
                   \text{ and \emph{for every} } k > \ell, \;
                    Exp \bigl[ M_{k}(\blds{\delta}^{\ell}) \setminus A_{\ell} ( \blds{\delta}^{\ell} ) \bigr]
                      \cap Exp \bigl[ A_{\ell} ( \blds{\delta}^{\ell} ) \bigr] = \varnothing. 
	\end{multline}
Note that \eqref{E:Exp.1-1.on.partial.union} holds trivially for $\ell = 0$ and $\delta_{i} := \delta^{0}_{i}$ ($i=1,2, \ldots$). This is the starting point of a recursive procedure for constructing $\blds{\delta}$ for which \eqref{E:Exp.1-1.on.union} holds. Moreover,
	\begin{multline}   \label{E:partial.union.monot.in.delta}
		\text{If \eqref{E:Exp.1-1.on.partial.union} holds then it also holds with } \\
		    \blds{\delta} \text{ replaced by any }
			\blds{\delta}' \leq \blds{\delta} \text{ (component-wise). }
	\end{multline}

Suppose that for every $i = 1, 2, \ldots$, we can find 
$\blds{\delta}^{i} \leq \blds{\delta}^{i-1}$ s.t.\ $\deli_{k} = \delta^{i-1}_{k}$ 
($k = 1, 2, \ldots, i-1$) and \eqref{E:Exp.1-1.on.partial.union} holds with $\ell = i$. Then \eqref{E:Exp.1-1.on.partial.union} holds for 
$\blds{\delta}^{\ell}$ ($\ell = 0, 1, 2, \ldots$) all equal to the diagonal, 
$\blds{\delta}: \delta_{i} = \deli_{i}$ ($i = 1, 2, \ldots$). In particular, \eqref{E:Exp.1-1.on.finite.union}, hence \eqref{E:Exp.1-1.on.union}, holds. This is because, for every $k = 1, 2, \ldots$, we have $\delta_{j} = \delta^{j}_{j} \leq \delta^{k}_{j}$ for $j \geq k$ and $\delta_{j} = \delta^{j}_{j} = \delta^{k}_{j}$ for $j \leq k$. Now apply \eqref{E:partial.union.monot.in.delta}.

Inductively, let $i = 1, 2, \ldots$ and suppose \eqref{E:Exp.1-1.on.partial.union} holds for $\ell = 0, \ldots, i-1$. Suppose $\blds{\delta} \leq \blds{\delta}^{i-1} \leq \cdots \leq \blds{\delta}^{0}$. Since $\delta_{i} \leq \delta^{0}_{i}$ we have that $Exp$ is one-to-one on $M_{i}$. Since $\blds{\delta} \leq \blds{\delta}^{i-1}$, we have, by the induction hypothesis and \eqref{E:partial.union.monot.in.delta}, that  $Exp$ is also one-to-one on $A_{i-1} \bigl( \blds{\delta} \bigr) \subset A_{i-1} \bigl( \blds{\delta}_{i-1} \bigr)$. Therefore, if $(y,v), (y',v') \in A_{i} \bigl( \blds{\delta} \bigr)$ are distinct but $Exp (y',v') = Exp (y,v)$, then we must have $(y,v) \in M_{i} \bigl( \blds{\delta} \bigr) \setminus A_{i-1} \bigl( \blds{\delta} \bigr)$, say, and $(y',v') \in A_{i-1} \bigl( \blds{\delta} \bigr) \setminus M_{i} \bigl( \blds{\delta} \bigr) \subset A_{i-1} \bigl( \blds{\delta} \bigr)$. But by the induction hypothesis, \eqref{E:partial.union.monot.in.delta} again, and the second part of \eqref{E:Exp.1-1.on.partial.union}, we have $Exp (y',v') \neq Exp (y,v)$. Therefore, for such 
$\blds{\delta}$, we have that $Exp$ is one-to-one on $A_{i} \bigl( \blds{\delta} )$.

By the induction hypothesis and \eqref{E:partial.union.monot.in.delta} again, we have that \eqref{E:Exp.1-1.on.partial.union} holds with $\ell = i-1$ and $\blds{\delta}^{\ell}$ replaced by $\blds{\delta}$. By local finiteness of the cover $\{ \clU_{k} \}$ there exists $I = I(i) \geq i$ s.t.\ $\text{dist} (\Pf \setminus \mcl{G}_{I}, \mcl{G}_{i}) \bigr\} > 0$. 
Suppose $\blds{\delta}$ satisfies 
	\begin{multline}  \label{E:delta.dist.ineq}
		\blds{\delta} \leq \blds{\delta}^{i-1}, \; \delta_{k} = \delta^{i-1}_{k} \, 	
		     (k = 1, 2, \ldots, i-1), \text{ and, } \\
		    \text{ if } j = i \text{ or } j > I(i), \text{ we have  } 
		      \delta_{j} \leq \frac{1}{3} \text{dist} (\Pf \setminus \mcl{G}_{I}, \mcl{G}_{i}) > 0.
	\end{multline}
Suppose $j > I$, $(x,u) \in M_{i}(\blds{\delta})$, and $(y,v) \in M_{j}(\blds{\delta})$. Then 
	\begin{equation*} \label{E:Exp.x.u.y.v.xi}
		\xi \bigl[ Exp(x,u), Exp(y,v) \bigr] \geq \xi(x,y) - |u| - |v| 
			> \text{dist} (\clU_{j}, \mcl{G}_{i}) - \delta_{i} - \delta_{j} > 0.
	\end{equation*}
Thus, we have
	\begin{multline}   \label{E:Exp.1-1.on.partial.union.i.I}
		Exp \text{ is one-to-one on } A_{i} ( \blds{\delta} ) \\
                   \text{ and \emph{for every} } k > I(i), \;
                    Exp \bigl[ M_{k}(\blds{\delta}) \setminus A_{\ell} ( \blds{\delta} ) \bigr]
                      \cap Exp \bigl[ A_{\ell} ( \blds{\delta} ) \bigr] = \varnothing. 
	\end{multline}

Suppose for no $\blds{\delta}$ satisfying \eqref{E:delta.dist.ineq} does \eqref{E:Exp.1-1.on.partial.union} hold with 
$\ell = i$ and $\blds{\delta}^{\ell} = \blds{\delta}$. Then, for example, \eqref{E:Exp.1-1.on.partial.union} does not hold if 
$\blds{\delta} = \tilde{\blds{\delta}}_{n}$, ($n = 1, 2, \ldots$), where $\tilde{\blds{\delta}}_{n}$ is any sequence s.t.\ $\blds{\delta} = \tilde{\blds{\delta}}_{n}$ satisfies \eqref{E:delta.dist.ineq}, and, for $j = i, \ldots, I(i)$, 
we have $\tilde{\delta}_{n,j} = \min \{ n^{-1}, \delta^{i-1}_{j} \}$.
Thus, for every $n = 1, 2, \ldots$, we have 
		\begin{multline}   \label{E:Exp.race.to.bottom}
			\text{For every } n = 1, 2, \ldots, \text{ there exists } 
			   (x_{n}, u_{n}) \in M_{i}(\tilde{\blds{\delta}}_{n}), \; j_{n} = i+1, \ldots, I(i),  \\
			     \text{ and } (y_{n}, v_{n}) \in M_{j_{n}}(\tilde{\blds{\delta}}_{n}) 
			          \setminus A_{i} ( \tilde{\blds{\delta}}_{n} )
			        \text{ s.t.\ } Exp(x_{n}, u_{n}) = Exp(y_{n}, v_{n}). 
	      \end{multline}
	      
By compactness of $\overline{\mcl{G}_{I}}$ (recall $\mcl{G}_{I} := \bigcup_{k=1}^{I} \clU_{k}$), there is an increasing sequence, $\{ n_{t} \}$, of positive integers s.t.\ $x_{n_{t}} \to x_{\infty} \in \overline{\mcl{G}_{i}}$ and $y_{n_{t}} \to y_{\infty} \in \overline{\mcl{G}_{I}}$ as $t \to \infty$.  
Then, $x_{\infty} = Exp(x_{\infty}, 0) = Exp(y_{\infty}, 0) = y_{\infty}$. Thus, for sufficiently large $n$, we have 
$x_{n}, y_{n} \in \mcl{V}_{x_{\infty}}$ and $|u_{n}|, |v_{n}| < \epsilon'_{x_{\infty}}$. If $x_{n} \neq y_{n}$ by \eqref{E:Exp.is.diffeo.on.Vx.eps'}, it is impossible for $Exp(x_{n}, u_{n}) = Exp(y_{n}, v_{n})$. So suppose $x_{n} = y_{n}$. Now, $(x_{n}, u_{n}) \in M_{i}(\tilde{\blds{\delta}}_{n}) \subset A_{i} ( \tilde{\blds{\delta}}_{n} )$ 
while $(x_{n}, v_{n}) = (y_{n}, v_{n}) \notin  A_{i} ( \tilde{\blds{\delta}}_{n} )$. 
Therefore, we must have $u_{n} \neq v_{n}$. But $|u_{n}|, |v_{n}| < \epsilon'_{x_{\infty}}$. Hence,
again by \eqref{E:Exp.is.diffeo.on.Vx.eps'}, it is impossible for $Exp(x_{n}, u_{n}) = Exp(y_{n}, v_{n})$. Therefore, \eqref{E:Exp.race.to.bottom} is false. 

Therefore, by \eqref{E:Exp.1-1.on.partial.union.i.I}, there exists $\blds{\delta}$ s.t.\ \eqref{E:delta.dist.ineq} holds and \eqref{E:Exp.1-1.on.partial.union} holds with $\ell = i$ and $\blds{\delta}^{\ell} = \blds{\delta}$. Therefore, by induction, \eqref{E:Exp.1-1.on.partial.union} holds for every $\ell = 0, 1, 2, \ldots$ and, hence, as explained above, the claim \eqref{E:Exp.1-1.on.union} holds.

Next, we prove \emph{claim:} Again reducing the $\delta_{j} \leq \delta^{0}_{j}$ ($j=1, 2, \ldots$) if necessary, we may assume the following. Let $\ell = 1, 2, \ldots$ be arbitrary. Let $(x', u) \in M_{\ell}(\blds{\delta})$ and let $x = Exp(x', u)$. If $y' \in \Pf$ and $y' \neq x'$ then 
	\begin{equation}   \label{E:xi.y'.x.>.|u|}
		r_{0} := \xi(y', x) > |u| = \xi(x', x).
	\end{equation} 
(Recall that $\xi$ is the topological metric on $\D$ corresponding to the Riemannian metric.) 

First, consider the case $r_{0} := \xi(y', x) = 0$. Then $Exp_{y'}(y', 0) = y' = x = Exp_{x'} (x', u)$. This contradicts the fact that $Exp$ is one-to-one on $A(\blds{\delta})$. So assume
	\begin{equation}  \label{E:r0.>.0}
		r_{0} > 0.
	\end{equation}
 
Let $\eta_{\ell} = \eta \in ( 0, \delta_{\ell}/2 )$ and suppose $(x',u) \in \hat{N}^{\eta} \restriction_{\clU_{\ell}}$. Now, by \eqref{E:Exp.is.diffeomorph.near.P}, $Exp_{x}$ is injective on the $\eta$-ball, $T_{x} \D(\eta)$, centered at 0 
in $T_{x} \D$. Let $y' \in \Pf$ and suppose $y' \neq x'$ but $r_{0} \leq |u| = \xi(x', x)$. 
Now $|u| < \eta$, so $r_{0} < \eta$ and $y' \in \overline{\mcl{B}_{\eta}(x)} := \bigl\{ z \in \D : \xi(z, x) \leq \eta \bigr\}$ (see \eqref{E:ball.defn}). Since $\Pf$ is locally compact, by \eqref{E:Pf.loc.compct.strat.space}, making $\eta$ smaller if necessary, we have 
	\begin{equation}  \label{E:B.eta.cap.P1.compact}
		\overline{\mcl{B}_{2 \eta}(x)} \cap \Pf \text{ is non-empty and compact}.
	\end{equation}
(Making $\eta$ smaller might entail making $u$ shorter and, hence, $x$ closer to $x'$.) Hence, by local finiteness, there exist only finitely many $i$ s.t.\ $\clU_{i} \cap \overline{\mcl{B}_{\eta}(x)} = \clU_{i} \cap \overline{\mcl{B}_{\eta}(x)} \cap \Pf \neq \varnothing$. It follows that we may also require $\eta > 0$ to satisfy $\eta < \delta_{i}$ for every $i$ s.t.\ $\clU_{i} \cap \overline{\mcl{B}_{\eta}(x)} \neq \varnothing$. 

Let $q(t)$ be a $C^{\infty}$ curve in $\Pf$ defined in an open interval $J \subset \RR$ about 0 s.t.\ $q(0) =  y'$ and  
let $r(t) := r_{q}(t) := \xi \bigl( x, q(t) \bigr)$ ($t \in J$). Thus, by \eqref{E:r0.>.0}, $r(0) = r_{0} > 0$. Making $J$ smaller, if necessary, we may assume $r(t) > 0$ ($t \in J$). Because, by \eqref{E:Exp.is.diffeomorph.near.P}, $Exp$ is a diffeomorphism in $T_{x'} \D(\epsilon'_{x'})$, there exists a curve $X(t) \in T_{x} \D$ ($t \in J$) s.t.\ $q(t) = Exp_{x} \bigl[ r(t) X(t) \bigr]$, ($t \in J$), viz.\ 
$X(t) = r(t)^{-1} Exp_{x}^{-1} \bigl[ q(t) \bigr]$. 
Thus, 
	\begin{equation}  \label{E:|X|.=.1}
		\bigl\| X(\cdot) \bigr\| \equiv 1,
	\end{equation} 
where $\| \cdot \|$ is the norm of the Riemannian metric.\footnote{We prove this. Given $t \in J$ let $\lambda = \lambda_{t}$ be the curve defined by $\lambda_{t}(v) = Exp_{x} \bigl[ v r(t) X(t) \bigr]$ ($v \in \RR$). Then $\lambda$ is a geodesic and $\lambda'(0)= r(t) X(t)$ (Boothby \cite[Lemma (6.4), p.\ 334]{wmB75}). 
By Boothby \cite[Lemma (5.2), p.\ 327]{wmB75}, arc length, $s(w) := \xi \bigl[ \lambda(w), x \bigr]$ is proportional to $w \in \RR$. But $r(t) = s(1)$. Hence, $s(w) = r(t)w$ ($w \in \RR$). Now, from Boothby \cite[p.\ 186]{wmB75}, we see that in addition $r(t) w = s(w) = \int_{0}^{w} \bigl\| \lambda'(v) \bigr\| \, dv$. (See Boothby \cite[Lemma (7.4), p.\ 341]{wmB75}) Now, $\lambda'$ is continuous (since $Exp$ is $C^{\infty}$). Therefore, taking the right hand derivative of both sides w.r.t.\ $w$ at 0, 
we get $r(t) = \bigl\| \lambda'(0) \bigr\| = \bigl\| r(t) X(t) \bigr\| = r(t) \bigl\| X(t) \bigr\|$. Since $r(t) > 0$, 
we get $\bigl\| X(t) \bigr\| = 1$, as desired.} 

Let $y(t)$, ($t \in [0, r_{0}]$) be the geodesic in $\D$ joining $x$ to $y'$. Write $\dot{y} := dy/dt$. We may assume 
$| \dot{y} | \equiv 1$. (So $y(0) = x$ and $y(r_{0}) = y'$.) Since $y(t)$ is a geodesic, $Exp$ is one-to-one 
on $T_{x'} \D(\epsilon'_{x'})$ and $y(r_{0}) = y' = q(0) = Exp_{x} r(0) X(0) = Exp_{x} r_{0} X(0)$, we have that 
$y(t) = Exp_{x} t X(0)$. 

We also have that $r(t)$ is differentiable near $t=0$. To see this, note that $r(t) X(t) = Exp_{x}^{-1} \bigl[ q(t) \bigr]$, 
so $r(t) = \Bigl\| Exp_{x}^{-1} \bigl[ q(t) \bigr] \Bigr\|_{x}$, where $\| \cdot \|_{x}$ is the norm of the Riemannian metric at $x$. Now, $Exp_{x}$ is a diffeomorphism on the $T_{x} \D(\eta)$, $q$ is differentiable, and $r(t) > 0$ near $t=0$. Hence, $r(t)$ is differentiable near $t=0$.

We have the following. (The proof is given below.)

	\begin{lemma}  \label{L:r'(0).<.0}
		If $r_{q}'(0) = 0$ for every $q$ then $\dot{y}(r_{0}) \perp T_{y'} \Pf$. 
	\end{lemma} 

Suppose $\dot{y}(r_{0}) \perp T_{y'} \Pf$, i.e., $\bigl\langle Y, \dot{y}(r_{0}) \bigr\rangle = 0$ for every $Y \in T_{y'} \Pf$. Let $z(t) := y(r_{0} - t)$, $t \in [0,r_{0}]$. Then $\dot{z}(0) = - \dot{y}(r_{0}) \perp T_{y'} \Pf$ and $z(r_{0}) = x$. 
Since $\bigl\| \dot{y}(r_{0}) \bigr\| = 1$, we have $|v| = r_{0} \leq |u| < \eta$, where $(y',v) := r_{0} \, \dot{z}(0)$. 
Thus, $(y',v) \in N$ and $Exp(y',v) = x = Exp(x',u)$. 
But $Exp$ is one-to-one on $A(\blds{\delta})$. Therefore, $y' = x'$. Contradiction.

Now suppose $\dot{y}(r_{0})$ is not perpendicular to $T_{y'} \Pf$. Then, by lemma \ref{L:r'(0).<.0}, we can choose $q$ s.t.\ $r_{q}'(0) < 0$. Hence, by nudging $y' \in \Pf$ along $q$ a little, we have $\xi(y', x) < \xi(x', x) \leq \eta$. For each $n$, pick $y_{n} \in \Pf$ s.t., 
	\begin{multline*}
		\xi(y_{n}, x) \leq \inf_{y \in \Pf} \xi(y, x) + 1/n = \inf_{y \in \mcl{B}_{\eta}(x) \cap \Pf} \xi(y, x) 
		 +1/n \leq \xi(y', x) + 1/n 
		   < \xi(x', x) + 1/n \\
		       (n = 1, 2, \ldots). 
	\end{multline*}
Then eventually we have $y_{n} \in \overline{\mcl{B}_{2 \eta}(x)} \cap \Pf$, hence, by \eqref{E:B.eta.cap.P1.compact}, there exists $y_{\infty} \in \overline{\mcl{B}_{2 \eta}(x)} \cap \Pf$ s.t.\ $\xi(y_{\infty}, x) = \inf_{y \in \Pf} \xi(y, x) < \xi(x', x) < \eta$. 
So, in fact $y_{\infty} \in \mcl{B}_{\eta}(x) \cap \Pf$. 
 
Thus, replacing $y'$ by $y_{\infty}$ we have $r_{q}'(0) = 0$ for any $C^{\infty}$ path $q$ in $\Pf$ s.t.\ $q(0) = y'$. (Because otherwise we could find $y'' \in \Pf$ even closer to $x$ than $y_{\infty}$ is.) Therefore, by lemma \ref{L:r'(0).<.0}, we have $y(t)$ is orthogonal to $T_{y'} \Pf$. 
But as we have seen, this means that $y_{\infty} = x'$. I.e., $\xi(x', x) < \xi(x', x)$. This absurdity completes the proof of \eqref{E:xi.y'.x.>.|u|}. \eqref{E:alpha.dist.to.P} follows. 
 
Let $\zeta_{1}, \zeta_{2}, \ldots $ be a $C^{\infty}$ partition of unity on $\Pf$ s.t.\ $supp \, \zeta_{k} \subset \Pf$ is compact for every $k = 1, 2, \ldots$. (See Spivak \cite[Theorem 15, p.\ 68]{mS79.SpivakVol1}, 
Boothby \cite[Theorem (4.4), p.\ 192]{wmB75}, and Helgason \cite[Theorem 1.3, p.\ 8]{sH62.SymSpaces}.) Let
	\begin{equation*}
		\epsilon_{i} 
		     := \min \bigl\{ \delta_{j} : (supp \, \zeta_{i}) \cap \clU_{j} \neq \varnothing \bigr\},
			\quad i = 1, 2, \ldots.
	\end{equation*}
Since $supp \, \zeta_{i}$ is compact and $\clU_{1}, \clU_{2}, \ldots$ is locally finite there are only finitely many $\clU_{j}$'s s.t.\ $(supp \, \zeta_{i}) \cap \clU_{j} \neq \varnothing$. Therefore, each $\epsilon_{i}$ is strictly positive. Define
	\begin{equation*}
		\epsilon_{\Pf}(x) := \sum_{k} \zeta_{k}(x) \, \epsilon_{k}, \quad x \in \Pf.
	\end{equation*}
Let $i = 1,2, \ldots$. Let $x \in \Pf$. Suppose $x \in \clU_{i}$. Then
	\begin{equation*}
		\epsilon_{\Pf}(x) = \sum_{x \in supp \, \zeta_{k}} \zeta_{k}(x) \, \epsilon_{k} 
		\leq \Bigl( \min_{x \in \clU_{j}} \delta_{j} \Bigr) \sum_{k} \zeta_{k}(x) \leq \delta_{i}.
	\end{equation*}
So $0 < \epsilon_{\Pf}(x) \leq \delta_{i}$. Therefore, $Exp$ is a diffeomorphism on
	\begin{equation*}
		\hat{N}^{\epsilon_{\Pf}} := \bigl\{ (x, v) \in N : x \in \Pf, |v| < \epsilon_{\Pf}(x) \bigr\}.
	\end{equation*}

Let $\mcl{C} := Exp \bigl[ \hat{N}^{\epsilon_{\Pf}} \bigr]$ and define $\alpha : \mcl{C} \to \hat{N}^{\epsilon_{\Pf}}$ by 
	\begin{equation} \label{E:alpha.=.Exp.invrs}
		\alpha := \bigl( Exp \restriction_{\hat{N}^{\epsilon_{\Pf}}} \bigr)^{-1}.
	\end{equation}
\end{proof}

  \begin{proof}[Proof of lemma \ref{L:r'(0).<.0}] 
Let $Y \in T_{y'} \Pf$. We may assume $q'(0) = Y$. By \eqref{E:xi.y'.x.>.|u|}, we have $r(0) = r_{0}$. Let $I$ be an open interval containing $[0, r_{0}]$ for which we can define $\tilde{q} : I \times J$ 
by $\tilde{q}(s, t) := Exp_{x} s X(t)$. Now, the curve $ s \mapsto \tilde{q}(s,0) = Exp_{x} \bigl( s X(0) \bigr)$ is a geodesic joining $x$ to $y'$. So is $y(\cdot)$. But $Exp$ is one-to-one on $T_{x'} \D(\epsilon'_{x'})$ 
and, by \eqref{E:|X|.=.1}, $\bigl\| X(\cdot) \bigr\| = 1 = \bigl\| \dot{y}(\cdot) \bigr\|$. Therefore, $y(s) = \tilde{q}(s,0)$ 
($s \in [0, r_{0}]$). Thus, 
$\tilde{q}(\cdot, 0)_{\ast} \left( \tfrac{\partial}{\partial s} \restriction_{s=r_{0}} \right) = \dot{y}(r_{0})$.
We have $\xi \bigl[ x, \tilde{q}(r_{0}, \cdot) \bigr] \equiv r_{0}$. (To see this, recall \eqref{E:|X|.=.1} and apply Boothby \cite[Lemma (7.4), p.\ 341]{wmB75} with $``r(v)'' := v r_{0}$, $``X(v)'' :\equiv X(t)$, $a=0$, and $b=1$.)
I.e., $\tilde{q}(r_{0}, \cdot)$ lies in the ``geodesic sphere'' 
$S_{r_{0}} := \bigl\{ Exp_{x} W \in \D : W \in T_{x} \D, \, \|W\| = 1 \bigr\}$. 
Thus, $Z_{y'} := \tilde{q}(r_{0}, \cdot)_{\ast} \left( \tfrac{\partial}{\partial t} \restriction_{t=0} \right) \in T_{y'} \D$ is tangent to $S_{r_{0}}$ at $y'$. Therefore, since $t_{0} < \eta$ so by \eqref{E:Exp.is.diffeomorph.near.P}, $Exp_{x}$ is a diffeomorphism on the $\eta$-ball, $T_{x} \D(\eta)$, about 0 in $T_{x} \D$, we have, by Boothby \cite[Lemma (7.3), p.\ 340]{wmB75}, 
	\begin{equation}  \label{E:Z.perp.y.dot.r0}
		Z_{y'} \text{ is perpendicular to }
			\dot{y}(r_{0}) = \tilde{q}(\cdot, 0)_{\ast} \left( \tfrac{\partial}{\partial s} \restriction_{s=r_{0}} \right).
	\end{equation}

Let $K(t) := \bigl( r(t), t \bigr) \in \RR^{2}$ ($t \in J$). (Recall that $r(t) := \xi \bigl[ x, q(t) \bigr]$ and $r(0) = r_{0} > 0$.) 
Then $q(t) = \tilde{q} \circ K(t)$ and 
	\begin{align}  \label{E:Y.in.trms.r.q.tild}
	    Y &= q'(0) \notag \\
	       &= (\tilde{q} \circ K)_{\ast} \left( \frac{d}{dt} \restriction_{t=0} \right) \notag \\
	       &= \tilde{q}_{\ast} \circ K_{\ast} \left( \frac{d}{dt} \restriction_{t=0} \right)  
	       &\intertext{  \qquad \qquad (Boothby \cite[Theorem (1.6), p.\ 109]{wmB75})}
	       &= \tilde{q}_{\ast} \left( r'(0) \frac{\partial}{\partial s} \restriction_{s=r_{0}, t=0} 
	       		+ \frac{\partial}{\partial t} \restriction_{s=r_{0}, t=0} \right) \notag \\
	       &= r'(0) \tilde{q}(\cdot, 0)_{\ast} \left( \frac{\partial}{\partial s} 
	           \restriction_{s=r_{0}} \right)
	       		+ \tilde{q}(r_{0}, \cdot)_{\ast} \left( \frac{\partial}{\partial t} 
			\restriction_{t=0} \right) \notag \\
	       &= r'(0) \dot{y}(r_{0}) + Z_{y'}. \notag
	\end{align}
Thus, if $r'(0) = 0$, we have $Y = Z_{y'} \perp \dot{y}(r_{0})$. I.e., $Y \perp \dot{y}(r_{0})$, by \eqref{E:Z.perp.y.dot.r0}. Since $Y \in T_{y'} \Pf$ is arbitrary we have $\dot{y}(r_{0}) \perp T_{y'} \Pf$. 
  \end{proof}
  
  \begin{proof}[Proof of  \eqref{E:gExp=Expg.star}] 
Let $\nabla$ be the Riemannian connection on $\D$ (Boothby \cite[Definition (3.2), p.\ 314]{wmB75}). By assumption $\langle \cdot, \cdot \rangle$ is $G$ invariant. Therefore, by the following lemma $\nabla$ is $G$-equivariant. (See below for proof. See Boothby \cite[Exercise 4, p.\ 321]{wmB75}.)
	\begin{lemma}  \label{L:connection.is.G.invariant}
Let $g \in G$. Since $g$ is invertible, so is $g_{\ast} : T \D \to T \D$. For $x \in D$, define
	\begin{equation*}
		\nabla^{g}_{X_{x}}( Y ) 
		  := g_{\ast}^{-1} \bigl[ \nabla_{g_{\ast} X_{x}}( g_{\ast} Y ) \bigr]
		         \in T_{x} \D,
	\end{equation*}
where $X_{x} \in T_{x} (\D)$ and $Y$ is a vector field on $\D$. 
Then $\nabla^{g} = \nabla$.
	\end{lemma} 
	
Let $X' := (x',v) \in T \D$, and let $g \in G$. 
Now, by definition (Boothby \cite[Definition (6.3), p.\ 333]{wmB75}) $Exp(X') = \gamma(1)$, where $\gamma : (a,b) \to \D$ ($a < 0 < 1 < b$) is a geodesic and $\gamma'(0) = X'$. (So $\gamma(0) = x'$.)
Notice that $g_{\ast} \gamma'(t) = (g \circ\gamma)'(t)$. In particular, $g_{\ast}(X') = (g \circ \gamma)'(0)$.
Thus, (Boothby \cite[p.\ 319 and Definition (5.1), p.\ 326]{wmB75}) $\nabla_{\gamma'(t)} \gamma'(\cdot) = 0$ for $t \in (a,b)$.\footnote{Okay, this is a bit sloppy. $\nabla_{\gamma'(t)}$ is a functional on the space of vector fields on all of $\D$, or at least on an open neighborhood of $\gamma(t)$, but not on $\gamma \bigl[ (a,b) \bigr]$. (See Boothby \cite[Lemma (3.4), p.\ 314]{wmB75}.) We patch this up as follows. If $\gamma$ is constant, so $\gamma'(t) \equiv 0$, then trivially $g \circ \gamma$ is a geodesic. So assume $\gamma$ is not constant, so $\gamma'(t) \neq 0$. Then $\gamma$ is an immersion of $(a,b)$ into $\D$. Let $t \in (a,b)$ be fixed. Therefore, by Boothby \cite[Theorem (4.12), p.\ 74]{wmB75}, there is an neighborhood $(a_{t}, b_{t}) \subset (a,b)$ of $t$ s.t.\ the restriction $\gamma \restriction_{(a_{t}, b_{t})}$ is an imbedding. 

Therefore (Boothby \cite[Theorem (5.5), p.\ 78]{wmB75}), $\gamma(t)$ has a coordinate neighborhood $(\varphi, \clU) \subset \RR^{d}$ s.t.\ $\varphi \bigl[ \gamma(a_{t}, b_{t}) \cap \clU \bigr] \subset \RR$, where we identify $\RR$ with the set $\bigl\{ (s, 0, \ldots, 0) \in \RR^{d} : s \in \RR \bigr\}$ and $\varphi (\clU)$ is an open cube, $C$. Write $\varphi (\clU) = C = (a', b') \times$ (some cube in $\RR^{d-1}$). We may assume $\varphi \circ \gamma(t) = 0 \in \RR^{d}$ and 
$\varphi^{-1} \restriction_{(a', b') \times \{0\}} = \gamma \restriction_{(a',b')}$. 

Define a vector field, $Z$, on $C$ as follows. $Z_{z} := \tfrac{\partial}{\partial u_{1}} \restriction_{u = z}$ ($z \in C$). 
Let $Y = \varphi_{\ast}^{-1}(Z)$. Then the restriction of $Y$ to $\gamma(a,b)$ is just $\gamma'$. 
Define $\nabla_{\gamma'(t)} \gamma'(\cdot) = \nabla_{\gamma'(t)} Y$. 
Then $\nabla_{\gamma'(t)} \gamma'(\cdot) = \tfrac{D}{ds} \gamma'(s) \restriction_{s=t}= 0$.}
Thus, by lemma \ref{L:connection.is.G.invariant}
	\[
		0 = \nabla^{g}_{\gamma'(t)}( \gamma' )
		      = g_{\ast}^{-1} \bigl[ \nabla_{g_{\ast} \gamma'(t)}( g_{\ast} \gamma' ) \bigr]
		   = g_{\ast}^{-1} \Bigl[ \nabla_{(g \circ \gamma)'(t)} 
		     \bigl( (g \circ \gamma)' \bigr) \Bigr].
	\]
Applying $g_{\ast}$ to both sides of the preceding, we get 
$\nabla_{(g \circ \gamma)'(t)}( (g \circ \gamma)'(t) ) = 0$. I.e., $g \circ \gamma$ is a geodesic. It is tangent to 
$(g \circ \gamma)(d/ds \restriction_{t=0}) = g_{\ast} \gamma'(0) = g_{\ast} X'$ at 0. Hence, 
	\begin{equation} \label{E:Exp.is.natural}
		g \circ Exp(X') = g \circ \gamma(1) 
		        = Exp \bigl[ g_{\ast}(X') \bigr].
	\end{equation}
Since $g^{\ast} \bigl( \langle \cdot, \cdot \rangle \bigr) = \langle \cdot, \cdot \rangle$, we have (by \eqref{E:Riem.metrics.on.D.on.Rk.same})
		\begin{equation} \label{E:dg.is.isometry}
			\bigl\| g_{\ast, x'}(x', v) \bigr\|_{g(x')} = |v|, \quad (x', v) \in T_{x'} \D.
		\end{equation}
	where $\bigl\| \cdot \bigr\|_{g(x')}$ is the norm corresponding 
	to $\langle \cdot, \cdot \rangle_{g(x')}$. 

Combining this with \eqref{E:dg.is.isometry}, we get \eqref{E:gExp=Expg.star}.   
  \end{proof}
  
  \begin{corly}  \label{C:geods.on.D.G-invar}
Recall 
$g^{\ast} \bigl( \langle \cdot, \cdot \rangle \bigr) = \langle \cdot, \cdot \rangle$. 
Let $\gamma : (a,b) \to \D$ ($a < 0 < 1 < b$) be a geodesic. Then, for every 
$g \in G$, is $g \circ \gamma$ is a geodesic of the same length as $\gamma$.
  \end{corly}

  \begin{proof}[Proof of \eqref{E:smooth.<.min}] 
Recall that $\epsilon_{\Pf}$ is strictly positive. Let $g \in G$, let $\epsilon^{g} := \epsilon_{\Pf} \circ g^{-1}$, 
and let 
    \begin{equation*}
      \clU_{g} := \bigl\{ x \in \Pf : \tfrac{1}{2} \epsilon^{g}(x) 
        < \min_{h \in G} \epsilon^{h}(x) \bigr\}.
    \end{equation*} 
Then $\{ \clU_{g} : g \in G \}$ is an open cover of $\Pf$. 
Let $\{ f_{g} : g \in G \}$ be a smooth partition of unity with $supp \, f_{g} \subset \clU_{g}$ ($g \in G$). 
(See Spivak \cite[Theorem 15, p.\ 68]{mS79.SpivakVol1}. Recall that in this section $\Pf$ is a smooth manifold.)

We use an idea that will be used in the proof of theorem \ref{T:if.lin.combo.on.F.then.can.rstrct.to.bad.sings}.
Let $g_{1}, \ldots, g_{m}$ be the elements of $G$. For $k = 1, \ldots, m$ write $f_{k} := f_{g_{k}}$ and let
   \[
      \bar{f}_{k} = m^{-1} \sum_{\ell=1}^{m} f_{\ell} \circ g_{\ell} \circ g_{k}^{-1}.
   \]  
We have, 
   \begin{equation}    \label{E:G.equivariance.of.f.bar.measure.chapter}
      \text{If } g_{k'} = g^{-1} \circ g_{k} \text{ then } \bar{f}_{k} \circ g = \bar{f}_{k'} 
         \quad (g \in G; \; k = 1, \ldots, m).
   \end{equation}
Note that for $i = 1, 2, \ldots, m$ and $g \in G$ we have 
    \begin{equation}  \label{E:g.and.supp.fi.measure.chapter}
        g(supp \, f_{i}) = supp \, (f_{i} \circ g^{-1}). 
    \end{equation}
    
We \emph{claim} that $\{ \bar{f}_{i} \}$ is a partition of unity on $\Pf$ s.t.\ for every $i$ we have $supp \, \bar{f}_{i} \subset \clU_{g_{i}}$.
Let $k = 1, \ldots, m$. Observe that, since $f_{i} \geq 0$ for every $i$,
    \begin{equation}  \label{E:supp.fbar.=.union.supps}
      supp \, \bar{f}_{k} = \bigcup_{\ell} supp \, (f_{\ell} \circ g_{\ell} \circ g_{k}^{-1} )
    \end{equation}
and, by \eqref{E:g.and.supp.fi.measure.chapter}, 
    \begin{equation}  \label{E:supp.fell.gell.gk.invrs}
          supp \, (f_{\ell} \circ g_{\ell} \circ g_{k}^{-1} ) = g_{k} \circ g_{\ell}^{-1} ( supp \, f_{\ell} )
             \subset g_{k} \circ g_{\ell}^{-1} \bigl( \clU_{g_{\ell}} \bigr).
    \end{equation}

We show
    \begin{equation}  \label{E:gk.gell.invrs.U.gell.subset.U.gk}
        g_{k} \circ g_{\ell}^{-1} \bigl( \clU_{g_{\ell}} \bigr) \subset \clU_{g_{k}}.
    \end{equation} 
(If \eqref{E:gk.gell.invrs.U.gell.subset.U.gk} is true, then it is immediate that it holds with, ``$\subset$'' replaced by ``$=$''.) Let $x \in \clU_{g_{\ell}}$. 
Then
    \begin{align*}
        \tfrac{1}{2} \epsilon^{g_{k}} \bigl[ g_{k} \circ g_{\ell}^{-1}(x) \bigr] &= 
            \tfrac{1}{2} \epsilon_{\Pf} \circ g_{k}^{-1} \circ g_{k} \circ g_{\ell}^{-1} (x) 
              = \tfrac{1}{2} \epsilon^{g_{\ell}} (x) \\
          &< \min_{h \in G} \epsilon^{h}(x) \\
          &= \min_{h \in G} \bigl[ \epsilon_{\Pf} \circ h^{-1} \circ g_{\ell} \circ g_{k}^{-1} 
            \circ g_{k} \circ g_{\ell}^{-1}(x) \bigr] \\
          &= \min_{h \in G} \Bigl( \epsilon_{\Pf} \circ (g_{k} \circ g_{\ell}^{-1} \circ h)^{-1} 
            \bigl[ g_{k} \circ g_{\ell}^{-1}(x) \bigr] \Bigr)
            = \min_{g \in G} \Bigl( \epsilon^{g} \bigl[ g_{k} \circ g_{\ell}^{-1}(x) \bigr] \Bigr).
    \end{align*}
Thus, $g_{k} \circ g_{\ell}^{-1}(x)  \in \clU_{g_{k}}$. This proves \eqref{E:gk.gell.invrs.U.gell.subset.U.gk}.
Therefore, by \eqref{E:supp.fbar.=.union.supps} and \eqref{E:supp.fell.gell.gk.invrs}, we have 
    	\begin{equation}  \label{E:supp.fk.bar.in.U.gk}
    		supp \, \bar{f}_{k} \subset \clU_{g_{k}}. 
    	\end{equation} 

To complete proof of the claim we must show that
   \begin{equation}   \label{E:bar.f.is.part.of.1.measure.chapter}
      \sum_{i \geq 1} \bar{f}_{i}(x) = 1 \text{ for every } x \in \D.
   \end{equation}
Since $\{ f_{i} \}$ is a partition of unity on $\Pf$, we have, on $\Pf$, 
   \begin{align*}
   \sum_{k \geq 1} \bar{f}_{k} &= m^{-1} \sum_{k=1}^{m} \sum_{\ell=1}^{m} f_{\ell} \circ g_{\ell} \circ g_{k}^{-1}  \\
           &= m^{-1} \sum_{\ell=1}^{m} \sum_{k=1}^{m} f_{\ell} \circ g_{\ell} \circ g_{k}^{-1}  \\
     &= m^{-1} \sum_{\ell=1}^{m} \sum_{g \in G} 
           f_{\ell} \circ g_{\ell} \circ (g^{-1} \circ g_{\ell})^{-1}  \\
     &= m^{-1} \sum_{g \in G} \sum_{\ell=1}^{m} f_{\ell} \circ g  
     = m^{-1} \sum_{g \in G} 1 = 1.
   \end{align*}
This proves \eqref{E:bar.f.is.part.of.1.measure.chapter} and the claim that $\{ \bar{f}_{i} \}$ is a partition of unity on $\Pf$ subordinate 
to $\{ \clU_{g_{i}} \}$.

Let $\epsilon := \tfrac{1}{2} \sum_{k = 1}^{m} \bar{f}_{k} \, \epsilon^{g_{k}}$. Then 
$\epsilon$ is smooth. We show that $G$ invariant. Let $g \in G$ and for 
$k = 1, \ldots, m$, let $g_{j(k)} := g^{-1} \circ g_{k}$, so $k \mapsto j(k)$ is a bijection. Then
    \begin{equation}  \label{E:eps.circ.g}
        \epsilon \circ g = \tfrac{1}{2} \sum_{k = 1}^{m} (\bar{f}_{k} \circ g) (\epsilon^{g_{k}} 
          \circ g).
    \end{equation}
Now, 
    \begin{equation}  \label{E:eps.gk.o.g.=.eps.g.jk}
        \epsilon^{g_{k}} \circ g = \epsilon_{\Pf} \circ g_{k}^{-1} \circ g = \epsilon_{\Pf} 
          \circ g_{j(k)}^{-1} = \epsilon^{g_{j(k)}} .
    \end{equation}
Similarly, by \eqref{E:G.equivariance.of.f.bar.measure.chapter}, 
    \begin{equation*}
        \bar{f}_{k} \circ g = m^{-1} \sum_{\ell=1}^{m} f_{\ell} \circ g_{\ell} \circ g_{k}^{-1} \circ g
          = m^{-1} \sum_{\ell=1}^{m} f_{\ell} \circ g_{\ell} \circ g_{j(k)}^{-1} = \bar{f}_{j(k)}.
    \end{equation*}
Thus, substituting this and \eqref{E:eps.gk.o.g.=.eps.g.jk} into \eqref{E:eps.circ.g} we get
    \begin{equation*}  
        \epsilon \circ g(x) 
          = \tfrac{1}{2} \sum_{k = 1}^{m} \bar{f}_{j(k)}(x) \, \epsilon^{g_{j(k)}}(x) 
            = \epsilon(x), \qquad x \in \Pf .
    \end{equation*}

Finally, by definition of $\clU_{g_{k}}$ and \eqref{E:supp.fk.bar.in.U.gk}, we have
    \begin{multline*}
        \epsilon(x) = \tfrac{1}{2} \sum_{k} \bar{f}_{k}(x) \, \epsilon^{g_{k}} (x)
          \leq \sum_{k} \bar{f}_{k}(x) \, \min_{h \in G} \epsilon^{h}(x) \\
            = \min_{h \in G} \epsilon^{h}(x) 
              = \min_{h \in G} \bigl\{ \epsilon_{\Pf} \circ h^{-1}(x) \bigr\} 
                = \min \bigl\{ \epsilon_{\Pf} \circ g(x), \, g \in G \bigr\}, 
                  \qquad x \in \clU_{g_{k}} ,
    \end{multline*}
as desired.
  \end{proof}

  \begin{proof}[Proof of lemma \ref{L:connection.is.G.invariant}] 
By Boothby \cite[Theorem (3.3), p.\ 314]{wmB75} it suffices to show that $\nabla^{g}$ has properties (1) through (4) in Boothby \cite[p.\ 313]{wmB75}.

Let $a, b : \D \to \RR$ be smooth. Let $X, X', Y, Y'$ be vector fields on $\D$. Let ``$\cdot$'' be used to sometimes denote point-wise multiplication of a vector field by a scalar function. First, we want to show the following. 
	\begin{equation*}
		\nabla^{g}_{aX+bX'}( Y )  = a \cdot \nabla^{g}_{X}( Y )  + b 
		  \cdot \nabla^{g}_{X'}( Y ).
	\end{equation*}
By the corresponding property for $\nabla$,
	\begin{align*}
		\nabla^{g}_{aX+bX'}( Y ) 
		    &= g_{\ast}^{-1} \bigl[ \nabla_{a g_{\ast} X + b g_{\ast} X'}( g_{\ast} Y ) \bigr]  \\
		    &= g_{\ast}^{-1} \bigl[ a \cdot \nabla_{g_{\ast} X}( g_{\ast} Y ) 
		         + b \cdot \nabla_{g_{\ast} X'}( g_{\ast} Y ) \bigr]  \\
		    &= a \cdot (\nabla^{g}_{X} Y )  + b \cdot (\nabla^{g}_{X'} Y ),
	\end{align*}
as desired.

Next, we prove
	\begin{equation} \label{E:nabla.g.X.aY+bY'.identity}
		\nabla^{g}_{X_{x}}( aY + bY' ) 
		    = a(x) ( \nabla^{g}_{X_{x}} Y )  + b(x) ( \nabla^{g}_{X_{x}} Y' ) 
	    + (X_{x} a) Y_{x} + (X_{x} b) Y'_{x}.
	\end{equation}
(E.g., $\nabla_{X_{x}} Y \in T_{x} \D$ by Boothby \cite[Corollary (3.5), p.\ 315]{wmB75}.) By definition, if $x \in \D$,
	\begin{equation} \label{E:nabla.g.X.aY+bY'.defn}
		\nabla^{g}_{X_{x}}( aY + bY' ) 
		    = g_{\ast}^{-1} \Bigl( \nabla_{g_{\ast} X_{x}} \bigl[ g_{\ast} (a  Y  + b Y') \bigr] \Bigr).
	\end{equation}
Let us parse this carefully. 
By Boothby \cite[Definition (2.6), p.\ 119; p.\ 150; and Theorem (1.2), p.\ 107]{wmB75}, we have
	\begin{align*}
		\bigl[ g_{\ast} (a  Y  + b Y') \bigr]_{g(y)} 
		   &= g_{\ast} \bigl[ a(y) Y_{y}  + b(y) Y'_{y} \bigr] \\
		   &= a(y) (g_{\ast} Y_{y})  + b(y) (g_{\ast} Y'_{y}) \\
		   &= a(y) (g_{\ast} Y)_{g(y)}  + b(y) (g_{\ast} Y')_{g(y)},
		      \quad y \in \D.
	\end{align*}
Therefore, 
	\begin{equation}   \label{E:g.star.aY+bY'.at.z}
		\bigl[ g_{\ast} (a  Y  + b Y') \bigr]_{z} = (a \circ g^{-1})(z) (g_{\ast} Y)_{z} 
		     + (b \circ g^{-1})(z) (g_{\ast} Y')_{z},   \quad z \in \D.
	\end{equation}
Now apply  Boothby \cite[property (2), Definition (3.1), p.\ 313]{wmB75} to \eqref{E:nabla.g.X.aY+bY'.defn} to get
  \begin{align}  \label{E:nabla.g.aY+bY'.expand}
    \nabla^{g}_{X_{x}}( aY + bY' )
	    &= g_{\ast}^{-1} \biggl( (a \circ g^{-1}) \bigl[ g(x) \bigr] \nabla_{g_{\ast X_{x}}}
	      ( g_{\ast} Y )  
	          + (b \circ g^{-1}) \bigl[ g(x) \bigr] \nabla_{g_{\ast} X_{x}}(g_{\ast} Y' ) \notag \\
	    &  \qquad \quad  + \bigl[ (g_{\ast} X_{x}) (a \circ g^{-1}) \bigr] (g_{\ast} Y_{x}) 
	           + \bigl[ (g_{\ast} X_{x}) (b \circ g^{-1}) \bigr] (g_{\ast} Y'_{x}) \biggr)  \\
	    &= a(x) ( \nabla^{g}_{X_{x}} Y )  + b(x) ( \nabla^{g}_{X_{x}} Y' ) \notag \\
	    &  \qquad \quad  + \bigl[ (g_{\ast} X_{x}) (a \circ g^{-1}) \bigr] Y_{x} 
	           + \bigl[ (g_{\ast} X_{x}) (b \circ g^{-1}) \bigr] Y'_{x}. \notag 
  \end{align}
Now, by Boothby \cite[Theorem (1.2), p.\ 107]{wmB75},
	\begin{equation*}
		(g_{\ast} X_{x}) (a \circ g^{-1}) = X_{x} \bigl[ g^{\ast} (a \circ g^{-1}) \bigr] 
		  = X_{x}(a).
	\end{equation*}
Similarly, $(g_{\ast} X_{x}) (b \circ g^{-1}) = X_{x} \bigl[ g^{\ast} (b \circ g^{-1}) \bigr] = X_{x}(b)$. Substituting into \eqref{E:nabla.g.aY+bY'.expand} yields \eqref{E:nabla.g.X.aY+bY'.identity}.

Next, we prove
	\begin{equation*}
		[X,Y] = \nabla^{g}_{X}(Y) - \nabla^{g}_{Y}(X).
	\end{equation*}
By the corresponding property for $\nabla$ and theorem Boothby \cite[Theorem (7.9), p.\ 154]{wmB75},
	\begin{align*}
		\nabla^{g}_{X}(Y) - \nabla^{g}_{Y}(X)
		   &= g_{\ast}^{-1} \bigl[  \nabla_{g_{\ast} X}(g_{\ast} Y) 
		                     - \nabla_{g_{\ast} Y}(g_{\ast} X) \bigr] \\
		   &= g_{\ast}^{-1} \bigl( [g_{\ast} X, g_{\ast} Y] \bigr) \\
		   &= [X,Y],
	\end{align*}
as desired.

Finally, we prove
	\begin{equation*}
		X \langle Y, Y' \rangle = \bigl\langle \nabla^{g}_{X} Y, Y' \bigr\rangle
		     + \bigl\langle Y, \nabla^{g}_{X} Y' \bigr\rangle.
	\end{equation*}
Let $x \in \D$. By $g$-invariance of $\langle \cdot, \cdot \rangle$, the corresponding property for $\nabla$, and the definition of $g_{\ast} X_{x}$ (Boothby \cite[Theorem (1.2), p.\ 107]{wmB75}),
	\begin{align}  \label{E:prop4.of.Riem.cnnctn}
	    \bigl\langle \nabla^{g}_{X} Y, Y' \bigr\rangle_{x}
		     + \bigl\langle Y, \nabla^{g}_{X} Y' \bigr\rangle_{x}
		  &= \Bigl\langle g_{\ast}^{-1} \bigl[ \nabla_{(g_{\ast} X)_{g(x)}}(g_{\ast} Y) \bigr], 
		    Y' \Bigr\rangle + \Bigl\langle Y, 
		            g_{\ast}^{-1} \bigl[ \nabla_{(g_{\ast} X)_{g(x)}}(g_{\ast} Y') \bigr] 
		            \Bigr\rangle \notag \\
		  &= \bigl\langle \nabla_{(g_{\ast} X)_{g(x)}}(g_{\ast} Y), g_{\ast} Y' \bigr\rangle
		     + \bigl\langle g_{\ast} Y, 
		            \nabla_{(g_{\ast} X)_{g(x)}}(g_{\ast} Y') \bigr\rangle \notag \\
		  &= (g_{\ast} X)_{g(x)} \langle g_{\ast} Y, g_{\ast} Y' \rangle \\
		  &= (g_{\ast} X_{x}) \langle g_{\ast} Y, g_{\ast} Y' \rangle \notag \\
		  &= X_{x} \langle g_{\ast} Y, g_{\ast} Y' \rangle_{\circ g} \notag.
	\end{align}
Here the function $\langle g_{\ast} Y, g_{\ast} Y' \rangle_{\circ g}$ is defined by
	\[
		y \mapsto \bigl\langle (g_{\ast} Y)_{g(y)}, (g_{\ast} Y')_{g(y)} 
		  \bigr\rangle_{g(y)} \; , \quad y \in \D.
	\]
But by the $g$ invariance of $\langle \cdot, \cdot \rangle$ (Boothby \cite[pp.\ 200--201]{wmB75}) we have
	\[
		\langle g_{\ast} Y, g_{\ast} Y' \rangle_{g(y)} 
		         = g^{\ast} \bigl( \langle \cdot, \cdot \rangle_{y} \bigr)(Y_{y}, Y'_{y})
		            = \langle Y_{y}, Y'_{y} \rangle_{y}, \quad y \in \D.
	\]
Substituting this into \eqref{E:prop4.of.Riem.cnnctn}, we get
	\[
		\bigl\langle \nabla^{g}_{X} Y, Y' \bigr\rangle
		    + \bigl\langle Y, \nabla^{g}_{X} Y' \bigr\rangle
		      = X_{x} \langle Y, Y' \rangle,
	\]
as desired.
  \end{proof} 

  \begin{proof}[Proof of lemma \ref{L:2eps.P.exists}]
 Let $\mcl{V} \subset \Pf$ be a neighborhood of the sort promised by part \ref{I:local.triv} of definition \ref{D:fibering.by.cones}. (Recall that in this book all neighborhoods are open.) Let $n = 1, 2, \ldots$; $\mcl{A}_{i}$; $\msf{L}_{i}$; 
and $h_{i} : \mcl{A}_{i} \times \msf{CL}_{i} \to \pi_{C}^{-1}(\mcl{A}_{i})$ ($i \in \NN_{n}$) also be as in the definition. Think of ``$\mcl{V}$'' as not only denoting the point set in $\Pf$, but all the $\mcl{A}_{i}$'s, $\msf{CL}_{i}$'s, and $h_{i}$'s as well.
Let $i \in \NN_{n}$. Let $\lambda_{i}$ be the metric on $\msf{CL}_{i}$. It satisfies \eqref{E:dist.bnd.on.CL}. Define $\xi \times \lambda_{i}$ as in \eqref{E:metric.on.Pf.x.CL} 
with $\lambda = \lambda_{i}$. Let $(y,v) \in C[\mcl{A}_{i}]$. By \eqref{E:Riem.metrics.on.D.on.Rk.same} and \eqref{E:xi+.from.2.metrics}, we have
        \begin{equation}  \label{E:xi+.dist.in.tangent.space}
           \| (y, v) - (y, w) \|_{y} = |v-w| = \xi_{+} \bigl[ (y, v), (y, w) \bigr],
             \quad (y, v), (y, w) \in T_{y} \D ,
        \end{equation}
where $| \cdot |$ is the norm on $\RR^{k}$, the ambient Euclidean space containing $\D$. Let $K' = K'_{i} < \infty$ be a common Lipschitz constant for $h_{i}$ and $h_{i}^{-1}$. Similarly, let $\bigl[ (s,z) \bigr] \in \msf{CL}_{i}$. Then, by \eqref{E:dist.bnd.on.CL},
        \begin{equation}  \label{E:lambda.dist.to.vertex}
          (\xi \times \lambda_{i}) 
            \left[ \Bigl(y, \bigl[ (s,z) \bigr] \Bigr), \Bigl(y, \bigl[ (0,z) \bigr] \Bigr) \right]
              = \lambda_{i} \left( \bigl[ (s, z) \bigr], \bigl[ (0, z) \bigr] \right) 
                = s \bigl| (1, z) \bigr| .
        \end{equation}

Replace $K'_{i}$ by 
	\begin{equation*}
		K := K_{i} := K' \max \biggl[ \sup \Bigl\{ \bigl| (1,w) \bigr| : w \in \msf{L}_{i} \Bigr\}, 
		  \inf \Bigl\{ \bigl| (1,w) \bigr| : w \in \msf{L}_{i} \Bigr\}^{-1} \biggr].
	\end{equation*}

Let $y \in \mcl{A}_{i}$, $s \in [0,1)$, $z \in \msf{L}_{i}$, and 
$(y,v) = h_{i}\bigl( \bigl[ (s,z) \bigr] \bigr)$. Then, by \eqref{E:lambda.dist.to.vertex}, \eqref{E:xi+.dist.in.tangent.space}, and Lipschitz property of $h_{i}^{-1}$ 
(part \ref{I:h.hi.invrs.Lip}  of definition \ref{D:fibering.by.cones}), 
	\begin{align*}  
	  s \inf \Bigl\{ \bigl| (1,w) \bigr| : w \in \msf{L}_{i} \Bigr\}
	    &\leq s \bigl| (1,z) \bigr| \\
	    &= \lambda_{i} \left( \bigl[ (s, z) \bigr], \bigl[ (0, z) \bigr] \right) \\
	    &= (\xi \times \lambda_{i}) 
            \left[ \Bigl(y, \bigl[ (s,z) \bigr] \Bigr), \Bigl(y, \bigl[ (0,z) \bigr] \Bigr) \right] \\
	    &\leq K_{i}' \xi_{+} \bigl[ (y,v), (y,0) \bigr]  = K_{i}' |v| .
	\end{align*}
In the other direction, using the Lipschitz property of $h_{i}$ we get,
	\begin{align*}  
	  |v| &= \xi_{+} \bigl[ (y,v), (y,0) \bigr] \\
	  &\leq K'_{i} (\xi \times \lambda_{i}) 
            \left[ \Bigl(y, \bigl[ (s,z) \bigr] \Bigr), \Bigl(y, \bigl[ (0,z) \bigr] \Bigr) \right] \\
	       &= K'_{i} \lambda_{i} \left( \bigl[ (s, z) \bigr], \bigl[ (0, z) \bigr] \right) \\
	       &= K'_{i} \, s \bigl| (1,z) \bigr| \\
	       &\leq  K'_{i} \, s \sup \Bigl\{ \bigl| (1,w) \bigr| : w \in \msf{L}_{i} \Bigr\} 
	         \leq K_{i} \, s  .
	\end{align*}
The net result is:
    \begin{multline}  \label{E:K.is.biLip.const}
       h_{i} \Bigl( y, \bigl[ (s, z) \bigr] \Bigr) = (y, v) \in C[\mcl{A}_{i}] 
         \text{ implies } K_{i}^{-1} s \leq |v| \leq K_{i} s, \\
            \quad \text{ for every } s \in [0,1), z \in \msf{L}_{i}, \text{ and } y \in \mcl{A}_{i}.
    \end{multline}
Define $\hat{\epsilon}_{i} := \tfrac{1}{2 K_{i}} < \infty$ and let 
$\tilde{\epsilon}_{\mcl{A}}(\mcl{V}) := \min \{ \hat{\epsilon}_{i} : i \in \NN_{n} \} \in (0, \infty)$. 
Let $t \in \bigl[ 0, \tilde{\epsilon}_{\mcl{A}}(\mcl{V}) \bigr]$, $y \in \mcl{V}$ 
and $(y,u) \in \mbf{F}_{1}[y]$, so $|u| = 1$. 
By definition of $\mbf{F}_{1}[y]$, \eqref{E:boldF.[E].defns}, there exists $b \in (0, 1]$ s.t.\ $(y, b t u) \in C[y]$. By definition \ref{D:fibering.by.cones}(\ref{I:CV.from.h.Ai}), there exists $i \in \NN_{n}$ and $\bigl[ (s, z) \bigr] \in \msf{CL}_{i}$ s.t.\  $y \in \mcl{A}_{i}$ and
$h_{i} \Bigl( y, \bigl[ (s, z) \bigr] \Bigr) = (y, b t u)$. In particular, by definition of $\msf{CL}_{i}$, $s \in [0,1)$. By \eqref{E:K.is.biLip.const}, 
$s \leq K_{i} b t |u| = K_{i} b t \leq K_{i}  b \hat{\epsilon}_{i} = b/2$. 
I.e., $2 s/b \leq 1$. Therefore, by \eqref{E:homogeneity.of.hi.plus} and the fact that $C[y]$ is a cone, 
	\begin{equation*}  
		(y, t u) = \frac{1}{b} \left( y, b t u \right) = 
		  \frac{1}{b} h_{i} \Bigl( y, \bigl[ (s, z) \bigr] \Bigr) 
		   = \frac{2 s}{b} h_{i} \Bigl( y, \bigl[ (1/2, z) \bigr] \Bigr) \in C[y].
	\end{equation*}

Let $\tilde{\mcl{V}} := \bigcup_{g \in G} g(\mcl{V})$. Here, we interpret ``$g(\mcl{V})$'' as the result of $g$ acting not just on the points of $\mcl{V}$ but on the whole structure, 
the $\mcl{A}_{i}$'s, $h_{i}$'s, etc., as in remark \ref{R:G.acts.on.V.etc}. So $\tilde{\mcl{V}}$ is $G$ invariant. Then, by remark \ref{R:G.acts.on.V.etc}, for every $g \in G$, we have that some $K''_{i} < \infty$ is a common Lipschitz constant for $gh_{i}$ and $(gh_{i})^{-1}$ for all $g \in G$. Use $K'_{i} := K''_{i}$ in the preceding construction. Define
$\tilde{\epsilon}_{\mcl{A}}(\tilde{\mcl{V}}) := \tilde{\epsilon}_{\mcl{A}}(\mcl{V})$. 
Thus, if $t \in \bigl[ 0, \tilde{\epsilon}_{\mcl{A}}(\tilde{\mcl{V}}) \bigr]$, $y \in \tilde{\mcl{V}}$  
and $(y,u) \in \mbf{F}_{1}[y]$, then $(y, tu) \in C[y]$.

Since $\{ \mcl{E}_{i} \}$ is locally finite, each $x \in \Pf$ 
has a neighborhood $\hat{\mcl{E}}$ 
s.t.\ for every $g \in G$, we have 
$g(\hat{\mcl{E}})$ 
intersects only finitely many $\mcl{E}_{i}$'s.\footnote{Let $x \in \Pf$. For $r > 0$, let 
$\D_{r}(x) := \bigl\{ y \in \D : \xi(x,y) < r \bigr\}$. (See \eqref{E:xi.is.metric.on.D}.) If no neighborhood, $\hat{\mcl{E}}$, having the desired property exists then, for every 
$m = 1, 2, \ldots$, 
there exist $i_{m} = 1, 2, \ldots$ and $g_{m} \in G$ s.t.\ $\mcl{E}_{i_{m}}, m = 1, 2, \ldots$ are distinct and $\mcl{E}_{i_{m}} \cap g_{m} \bigl( \D_{1/m}(x) \bigr) \neq \varnothing$ 
for every $m$. Since $G$ is finite WLOG $g_{m} \equiv g_{1} \in G$.  
But, by \eqref{E:xi.is.G-invar}, 
$g_{1} \bigl[ \D_{1/m}(x) \bigr] = \D_{1/m} \bigl[ g_{1} (x) \bigr]$. This contradicts the assumption that $\{ \mcl{E}_{i} \}$ is locally finite.}
Replace $\hat{\mcl{E}}$ by $\tilde{\mcl{E}} := \bigcup_{g \in G} g(\hat{\mcl{E}})$. 
Then $\tilde{\mcl{E}}$ is $G$-invariant, but it remains the case that $\tilde{\mcl{E}}$ intersects only finitely many $\mcl{E}_{i}$'s.
Let $\tilde{\epsilon}_{\mcl{E}}(\tilde{\mcl{E}}) := \min \bigl\{ t_{i} : \mcl{E}_{i} \cap \tilde{\mcl{E}} \neq \varnothing \bigr\}$. 
Thus, $\tilde{\epsilon}_{\mcl{E}}(\tilde{\mcl{E}})$ is strictly positive. If $y \in \tilde{\mcl{E}}$ 
then $\tilde{\epsilon}_{\mcl{E}}(\tilde{\mcl{E}}) \leq t_{i}$ for every $i$ and $g \in G$ s.t.\ 
$y \in g(\mcl{E}_{i})$.
 
Suppose $\tilde{\mcl{V}} \cap \tilde{\mcl{E}} \neq \varnothing$. Let 
$\tilde{\epsilon}(\tilde{\mcl{V}}, \tilde{\mcl{E}}) := \min \bigl\{ \tilde{\epsilon}_{\mcl{A}}(\tilde{\mcl{V}}), \tilde{\epsilon}_{\mcl{E}}(\tilde{\mcl{E}}) \bigr\}$. Then $\tilde{\epsilon}(\tilde{\mcl{V}}, \tilde{\mcl{E}}) \in (0, \infty)$ and 
	\begin{multline}  \label{E:eps.tilde.property}
		\text{If } t \in [0, \tilde{\epsilon}(\tilde{\mcl{V}}, \tilde{\mcl{E}})] 
		  \text{ and } (y, u) \in \mbf{F}_{1}[\tilde{\mcl{V}} \cap \tilde{\mcl{E}}], \\
		    \text{ then } (y, t u) \in C[y] \text{ and } t \leq t_{i} \text{ for every } i 
		      \text{ and } g \in G \text{ s.t.\ } y \in g(\mcl{E}_{i}).
	\end{multline}

By \eqref{E:Pf.has.countable.loc.finite.refinement}, there exist a countable locally finite refinement $\{ \X_{i}, i = 1, 2, \ldots \}$ of 
$\{ \tilde{\mcl{V}} \cap \tilde{\mcl{E}} : \tilde{\mcl{V}} \cap \tilde{\mcl{E}} \neq \varnothing \}$ (each $\X_{i}$ open). Since $\tilde{\mcl{V}}$ and $\tilde{\mcl{E}}$ are each $G$-invariant, so is $\tilde{\mcl{V}} \cap \tilde{\mcl{E}}$. Hence, we may assume each 
$\X_{j}$ is $G$-invariant. For each $j$, let 
	\begin{equation*}
	  \tilde{\epsilon}_{j} 
	    := \tfrac{1}{2} \sup \bigl\{  \tilde{\epsilon}(\tilde{\mcl{V}}, \tilde{\mcl{E}}) > 0 : 
	      \X_{j} \subset \tilde{\mcl{V}} \cap \tilde{\mcl{E}} \bigr\}.
	\end{equation*}
Then, by \eqref{E:eps.tilde.property}, for every $j = 1, 2, \ldots$, 
we have that $\tilde{\epsilon}_{j} > 0$ and
	\begin{multline}  \label{E:eps.hat.X.property}
		\text{If } t \in [0, \tilde{\epsilon}_{j}] \text{ and } (y, u) 
		  \in \mbf{F}_{1}[\X_{j}], \\
		    \text{ then } (y, t u) \in C[y] \text{ and } t \leq t_{i} \text{ for every } 
		      i \text{ and } g \in G \text{ s.t.\ } y \in g(\mcl{E}_{i}).
	\end{multline}
	
Recall that $\pi : T \D \to \D$ is projection and $\clU \subset T \D$ is the total space of the cone bundle, a subset of $T \D$. For each $j$ let
	  \begin{equation*}
	  	\mcl{C}_{j} := Exp \bigl( C[\X_{j}] \bigr) 
		  = Exp \bigl( \pi^{-1}(\X_{j}) \cap \clU \bigr)
		    = Exp \bigl( \pi_{C}^{-1}(\X_{j}) \bigr) .
	  \end{equation*}
Since $\pi$ is continuous, by \eqref{E:pi.is.smooth.and.open}, and $Exp$ is open on $C[\Pf]$ by part \ref{I:Exp.alpha.homeom} of definition \ref{D:fibering.by.cones}, $\mcl{C}_{j}$ is open. By \eqref{E:gExp=Expg.star}, 
part \ref {I:C.C[P].G.invar} of definition \ref{D:fibering.by.cones},  
and $G$-invariance of $\X_{j}$, we have that $\mcl{C}_{j}$ is $G$-invariant.
	  
Recall from part \ref{I:Exp.alpha.homeom} of definition \ref{D:fibering.by.cones} again, 
$\mcl{C} := Exp \bigl( C[\Pf] \bigr) = \bigcup_{j} \mcl{C}_{j} \subset \D$. 
Thus, $\mcl{C}$ is a $d$-dimensional $G$-invariant manifold with open cover 
$\{ \mcl{C}_{j} \}$. By Boothby \cite[Theorem (4.4), p.\ 192]{wmB75} or Spivak \cite[Theorem 15, p.\ 68]{mS79.SpivakVol1}, there exists
a $C^{\infty}$ partition of unity $\{ f_{j}, \; j = 1, 2, \ldots \}$ on $\mcl{C}$ 
s.t.\ $supp \, f_{j} \subset \mcl{C}_{j}$ ($j = 1, 2, \ldots$). For every $j$ and $g \in G$, we have
$supp \, f_{j} \circ g \subset \mcl{C}_{j}$, since $\mcl{C}_{j}$ is $G$-invariant, so by averaging $f_{j} \circ g$ over $g \in G$, we can make each $f_{j}$ $G$-invariant. 
Let $\zeta_{j}$ be the restriction, $f_{j} \restriction_{\Pf}$, of $f_{j}$ 
to $\Pf$. (So $\zeta_{j}$ is continuous and $G$-invariant since $f_{j}$ is.) 
Thus, $supp \, \zeta_{j} \subset \X_{j}$.
 
Define
	\begin{equation*}
	     \epsilon(x) := \frac{1}{3} \sum_{j} \zeta_{j}(x) \, \tilde{\epsilon}_{j}, \quad x \in \Pf.
	\end{equation*}
Notice that, since the $\zeta_{j}$'s are $G$-invariant, so is $\epsilon$. Moreover, $\epsilon$ has a $C^{\infty}$ extension $\mcl{C}$, 
viz.\ $\tfrac{1}{3} \sum_{j} f_{j} \, \tilde{\epsilon}_{j}$, which is also $G$-invariant.

Let $y \in \Pf$, let $(y,v) \in \mbf{F}_{1}[y]$, and let $t \in \bigl[ 0, \epsilon(y) \bigr]$. There exists $m = 1, 2, \ldots$ s.t.\ there are exactly 
$m$ indices $j_{1}, \ldots, j_{m}$ s.t.\ $\zeta_{j_{\ell}}(y) \neq 0$ 
for $\ell = 1, \ldots, m$. Thus, $y \in \X_{j_{\ell}}$ for $\ell = 1, \ldots, m$. 
Let $q = 1, \ldots, m$ satisfy 
$\tilde{\epsilon}_{j_{q}} = \max_{\ell = 1, \ldots, m} \tilde{\epsilon}_{j_{\ell}}$. Hence, $y \in \X_{j_{q}}$ and we have
	\begin{equation*}
		3 \, \epsilon(y) = \sum_{k=1}^{m} \zeta_{j_{k}}(y) \, \tilde{\epsilon}_{j_{k}} 
			\leq \tilde{\epsilon}_{j_{q}}. 
	\end{equation*}
Therefore, by \eqref{E:eps.hat.X.property} with $j = j_{q}$ and $t = 3 \, \epsilon(y)$, 
if $v \in \mbf{F}_{1}[y]$ then $\bigl (y, 3 \, \epsilon(y) v \bigr) \in C[y]$ and 
$2 \, \epsilon(y) < t_{i}$ for every $i$ s.t.\ $y \in \mcl{E}_{i}$, as desired.
  \end{proof}
  
  \begin{proof}[Proof of lemma \ref{L:rescaling.C[]}.]
Let $\clU \subset T \D \restriction_{\Pf}$ be the total space of a cone bundle as described in definition \ref{D:fibering.by.cones}. By hypothesis, we may assume the bundle has a relatively compact trivialization $\bigl\{ ( \mcl{V}, \mcl{A}_{i}, h_{i}, \msf{L}_{i} ) \bigr\}$. Let $\mcl{C} := Exp(\clU)$ as in part \ref{I:Exp.alpha.homeom} of definition \ref{D:fibering.by.cones}.

By Boothby \cite[Lemma (6.1), p.\ 332]{wmB75}, $\clU$ is a finite dimensional manifold. Therefore, by Boothby \cite[p.\ 52]{wmB75} 
or Munkres \cite[pp.\ 3--4]{jrM66}, it is locally compact, second countable, and Hausdorff. Therefore, as in example \ref{Ex:nested.compacts}, there exists a sequence 
$\mcl{K}_{0} = \varnothing, \mcl{K}_{1}, \mcl{K}_{2}, \ldots$ of compact subsets of $\clU$ whose union is $\clU$ that satisfy 
$\mcl{K}_{i-1} \subset \mcl{K}_{i}^{\circ}$ ($i = 1, 2, \ldots$). 
Let $\mcl{L}_{i} := \pi(\mcl{K}_{i}) \subset \Pf$ ($i = 0, 1, 2, \ldots$), 
where $\pi : T \D \to \D$ is projection. By \eqref{E:pi.is.smooth.and.open}, $\pi$ is continuous and open. Therefore, each $\mcl{L}_{i}$ is compact and 
$\mcl{L}_{i-1} \subset \mcl{L}_{i}^{\circ}$ ($i = 1, 2, \ldots$). 
Let $\mcl{E}'_{i} := \mcl{L}_{i+1}^{\circ} \setminus \mcl{L}_{i-1}$ ($i = 1, 2, \ldots$).
As in example \ref{Ex:nested.compacts}, $\{ \mcl{E}'_{i} \}$ is a locally finite open cover 
of $\Pf$ consisting of sets relatively compact in $\Pf$.

Since each $\mcl{E}'_{j}$ is relatively compact in $\Pf$, by \eqref{E:C[K].rel.compact} for each $j$ there exists $M_{j} \in (0, \infty)$ s.t.\ $|v| < M_{j}$ for every $(y,v) \in C[\mcl{E}'_{j}]$. By lemma \ref{L:2eps.P.exists}, there exists a positive $G$-invariant function $\epsilon_{1/M}$ 
on $\Pf$, with $C^{\infty}$ extension to $\mcl{C}$, s.t.\ $2 \epsilon_{1/M}(y) < 1/M_{j}$ for every $y \in \mcl{E}'_{j}$. 

Let $\mcl{E}''_{ij} := \mcl{E}_{i} \cap \mcl{E}'_{j}$. Obviously, $\{ \mcl{E}''_{ij} \}$ is locally finite. By lemma \ref{L:2eps.P.exists}, there exists a positive $G$-invariant function $\epsilon$ 
on $\Pf$, with $C^{\infty}$ extension to $\mcl{C}$, s.t.\ $2 \epsilon(y) < \min(t_{i}, M_{j})$ for every $y \in \mcl{E}''_{ij}$. Therefore, if $y \in \mcl{E}''_{ij}$ then 
    \begin{equation}  \label{E:eps.eps1.<.1}
      \epsilon(y) \epsilon_{1/M}(y) < \min(t_{i}, M_{j}) / M_{j} \in (0,1). 
    \end{equation}

Let $\hat{\epsilon}$ denote the $C^{\infty}$ extension of $\epsilon$ to $\clU$. 
Define $\hat{\epsilon}_{1/M}$ similarly. $T \D$ is a Riemannian manifold with Riemannian metric being the restriction of that given by \eqref{E:Riem.metric.on.T.R^k}. 
$\omega_{\D}$, defined at \eqref{E:omega.D.defn}, is the topological metric on $\clU$ corresponding to this Riemannian metric. (See \eqref{E:Riem.metric.on.TD} and lemma \ref{L:xi+.generates.manif.topol}.) By corollary \ref{C:cont.diff.=.loc.Lip} 
$\hat{\epsilon} : \clU \to \RR$ is locally Lipschitz w.r.t.\ $\omega_{\D}$ and the Euclidean norm on $\RR$. Hence, by \eqref{E:omegaD.leq.Kb.xi+}, $\hat{\epsilon} : \clU \to \RR$ is locally Lipschitz w.r.t.\ $\xi_{+}$ and the Euclidean norm on $\RR$. The same thing is true 
of $\hat{\epsilon}_{1/M}$. In particular,
    \begin{equation}  \label{E:eps.eps.1/M.loc.Lip}
      \epsilon \text{ and } \epsilon_{1/M} \text{ are locally Lipschitz w.r.t.\ } \xi_{+} 
        \text{and the Euclidean norm on } \RR.
    \end{equation}

Let the neighborhoods $\mcl{V}$ with finite covering sets, links, and imbeddings 
$\{ \mcl{A}_{i} , \msf{L}_{i}, h_{i} \}$ be as in definition \ref{D:fibering.by.cones}, part \ref{I:local.triv}. We modify this as follows. The $\mcl{V}$s,
$\mcl{A}_{i}$s, and $\msf{L}_{i}$s remain the same, but replace $h_{i}$ by 
    \begin{equation*}
      h_{i, \epsilon} : (y, s, sz) \mapsto \epsilon(y) \epsilon_{1/M}(y) h_{i}(y, s, sz),
        \qquad y \in \mcl{A}_{i}, s \in [0,1), z \in \msf{L}_{i} .
    \end{equation*}

Define 
$\clU_{\epsilon} := \epsilon \epsilon_{1/M} \clU 
:= \bigl\{ \epsilon(y) \epsilon_{1/M}(y) (y,v) \in T \D :  (y,v) \in \clU \bigr\}$, 
$C_{\epsilon}[y] =  \epsilon(y) \epsilon_{1/M}(y) C[y]$ ($y \in \Pf$), and
$\pi_{\epsilon} := \pi \restriction_{\clU_{\epsilon}}$. If $\X \subset \Pf$, 
define $C_{\epsilon}[\X] := \pi_{\epsilon}^{-1}(\X)$. So hanging the subscript $\epsilon$ appropriately we see that \eqref{E:pi.-1.V.=.union.of.As} holds. Moreover, it is immediate from the definition of cone and \eqref{E:eps.eps1.<.1} that $\clU_{\epsilon} \subset \clU$ and 
$C_{\epsilon}(y) \subset C[y]$. Hence, by part \ref{I:Exp.alpha.homeom} 
of definition \ref{D:fibering.by.cones} $Exp \restriction_{\clU_{\epsilon}}$ is a bi-Lipschitz homeomorphism. By lemmas \ref{L:scalar.mult.xi+.cont} and \ref{L:xi+.generates.manif.topol}, $\clU_{\epsilon}$ is open in $T \D$. So $\mcl{C}_{\epsilon} := Exp(\clU_{\epsilon})$ is open and part \ref{I:Exp.alpha.homeom} of definition \ref{D:fibering.by.cones} continues to hold in the subscript $\epsilon$ world.

By part \ref{I:local.triv} of definition \ref{D:fibering.by.cones}, the fact that the tivialization of the bundle is relatively compact by assumption, \eqref{E:eps.eps.1/M.loc.Lip},  
and example \ref{Ex:ratnl.fns.loc.Lip}, $h_{i, \epsilon}$ is Lipschitz. 

Let $y \in \mcl{A}_{i}$, $s \in [0,1)$, $z \in \msf{L}_{i}$ and $(y,v) =  h_{i, \epsilon}(y, s, sz)$, so $\left( y, \tfrac{1}{\epsilon(y) \epsilon_{1/M}(y)} v \right) \in C[y]$. $\mcl{A}_{i}$ is relatively compact by assumption and $\epsilon \epsilon_{1/M}$ is positive. 
Therefore, $\epsilon \epsilon_{1/M}$ is bounded below on $\mcl{A}_{i}$. By corollary \ref{C:cont.diff.=.loc.Lip} and example \ref{Ex:ratnl.fns.loc.Lip} again each step in the following is Lipschitz.
    \begin{equation*}
          \begin{CD}
            (y,v) @>{\bigl((\epsilon \epsilon_{1/M}) \circ \pi \bigr) \times identity}>>
              \bigl( \epsilon(y) \epsilon_{1/M}(y), y, v \bigr) @>>>
                 \left( y, \frac{1}{\epsilon(y) \epsilon_{1/M}(y)} v \right)
                   @>h_{i}^{-1}>> (y, s, sz) .
           \end{CD}
   \end{equation*}
Therefore, by \eqref{E:comp.of.Lips.is.Lip}, $h_{i, \epsilon}^{-1}$ is Lipschitz. 

Let $(y,v) \in \clU_{\epsilon}$. 
Thus, $\left( y, \tfrac{1}{\epsilon(y) \epsilon_{1/M}(y)} v \right) \in \clU$.
There exists $i,j$ s.t.\ $y \in \mcl{E}''_{ij}$. Therefore, $y \in \mcl{E}'_{j}$ so by definition of $M_{j}$, $\left| \tfrac{1}{\epsilon(y) \epsilon_{1/M}(y)} v \right| < M_{j}$. 
In addition, $y \in \mcl{E}_{i}$ so $\epsilon(y) < t_{i}$. Hence, by definition 
of $\epsilon_{1/M}$, we have 
    \begin{equation*}
     |v| = \epsilon(y) \epsilon_{1/M}(y) 
       \left| \frac{1}{\epsilon(y) \epsilon_{1/M}(y)} v \right| 
         \leq \epsilon(y) M_{j}^{-1} \left| \frac{1}{\epsilon(y) \epsilon_{1/M}(y)} v \right| 
           < \epsilon(y) < t_{i} .
    \end{equation*}
Thus, \eqref{E:v.short} holds in $\clU_{\epsilon}$. 

Let $(y, v) \in \clU_{\epsilon}$. 
Thus, $\left( y, \tfrac{1}{\epsilon(y) \epsilon_{1/M}(y)} v \right) \in \clU$. Let $g \in G$. Then, by \eqref{E:C.C[P].G.invar}, we have
        \begin{equation*}
          g_{\ast} (y,v) 
            = \epsilon(y) \epsilon_{1/M}(y) \,
              g_{\ast} \left( y, \tfrac{1}{\epsilon(y) \epsilon_{1/M}(y)} v \right) \in \clU_{\epsilon} .
        \end{equation*}
It follows from \eqref{E:gExp=Expg.star} that \eqref{E:C.C[P].G.invar} holds 
in the ``subscript $\epsilon$'' world.

It is obvious that the analogues of properties \ref{I:pi.h.(y,w)=y} and \ref{I:homogeneity.of.hi} of definition \ref{D:fibering.by.cones} hold for $h_{\epsilon}$

This completes the proof of the lemma.
  \end{proof}
  
  \begin{proof}[Proof of lemma \ref{L:biLip.triangulation}]
 $\D$ is a $C^{\infty}$ manifold. Therefore, by Munkres \cite[Theorem 10.6, pp.\ 103--104]{jrM66},
$\D$ has a $C^{\infty}$ triangulation $f : |P| \to \D$. That $f$ is a triangulation means that, in particular, $f$ is a homeomorphism (appendix \ref{Chptr:basics.of.simp.comps}). Therefore, since $\D$ is a compact by \eqref{E:D.is.cmpct.Riem.manif}, $|P|$ is compact, hence, by \eqref{E:|P|.compact}, finite. 

``$f : |P| \to \D$ is $C^{\infty}$'' means the following (Munkres \cite[Definition 8.3, pp.\ 80--81 and Definition 1.2, p.\ 5]{jrM66}). Recall that, by \eqref{E:D.imbedded.in.Rk}, $\D$ is an imbedded submanifold of $\RR^{k}$. Let $b \in |P|$ and let $\sigma \in P$ satisfy $b \in \sigma$. Then the map $x \mapsto D f \restriction_{\sigma} (b)(x-b) \in \RR^{k}$ is required to be one-to-one in $x \in \sigma$. What does $D f \restriction_{\sigma} (b)$ mean? If $\dim \sigma = \ell$ then, by \eqref{E:formla.for.affine.plane}, $\sigma$ lies in a unique $\ell$-plane in $\RR^{N}$ so WLOG we may assume $\sigma \subset \RR^{\ell}$. Then there exists a neighborhood, $U$, of $b$ in $\RR^{\ell}$ and a $C^{\infty}$ extension, $g : U \to \RR^{k}$, of $f$. Then $Df(b)$ is defined to be the Jacobian matrix 
$Dg(b)^{k \times \ell} := \bigl( \partial g_{i}/\partial x_{j} \bigr)$. Thus, 
``$x \mapsto D f \restriction_{\sigma} (b)(x-b) \in \RR^{k}$ is one-to-one in 
$x \in \sigma$'' means $\ell \leq k$ and $Dg(b)$ has rank $\ell$. Therefore, making $U$ smaller if necessary, $Dg$ has full rank $\ell$ throughout $U$. 

Hence, $N := g(U)$ is an immersed submanifold of $\RR^{k}$ (Boothby \cite[Definition (4.3)), p.\ 70]{wmB75}) of dimension $\ell$. Therefore, by Boothby \cite[Theorem (4.12)), p.\ 74]{wmB75}, making $U$ smaller if necessary, $g : U \to \RR^{k}$ is an imbedding. Hence, by Boothby \cite[Theorem (5.5), p.\ 78]{wmB75}, $g : U \to g(U)$ is a diffeomorphism. By corollary \ref{C:cont.diff.=.loc.Lip}, that means making $U$ even smaller if necessary, $g$ and $g^{-1}$ are both Lipschitz. 
In particular, $f \restriction_{U \cap \sigma} = g \restriction_{U \cap \sigma}$ and its inverse are both Lipschitz. Hence, by compactness 
of $\sigma$, $f$ is bi-Lipschitz on $\sigma$. But $P$ is finite. Therefore, $f$ is bi-Lipschitz on $P$.
  \end{proof}
  
 \begin{proof}[Proof of lemma \ref{L:f.dil.Lipschitz}]
Most of the work will go into proving
      \begin{multline}  \label{E:F.and.F.invrs.Lip}
	\text{With } R \in (0, \threeess/3) \text{ fixed, } 
	  F: = F_{R} 
	    \text{ and its inverse are Lipschitz on } C[\Pf] \text{ w.r.t. } \xi_{+}. \\
	      \text{Moreover, the Lipschitz constant for } F_{R} 
	        \text{ is inversely proportional to } R.
     \end{multline}
(See \eqref{E:xi+.from.2.metrics}.)     

By assumption, $R < \threeess/3$. Temporarily release $t$ from its definition 
as $\dim \T$. Let    
    \begin{equation}  \label{E:t:=.threeess/R}
        t := \threeess/R > 3 .
    \end{equation}
Thus, Lipschitz constants proportional to $t$ are inversely proportional to $R$. 
Write $A := A_{dilate,R}$ and $B := B_{dilate,R}$. Notice, by \eqref{E:AB.dilate.defn}, 
    \begin{equation}  \label{E:A.B.in.trms.of.t}
        A = \frac{t}{2t-1} \text{ and } B = 2 \frac{t-1}{2t-1}.
    \end{equation}

There are a lot of cases to be considered. That makes the proof rather lengthy. Notice that, by \eqref{E:Cs.defn} and \eqref{E:rho.in.[0,1]}, we have $C_{t^{-1}}$ is defined. First we consider \emph{Case 0:} The restrictions of $F$ to three subsets of $C[\Pf]$:  $C_{t^{-1}}$, 
$C_{\twoess} \setminus C_{t^{-1}}$, and $C[\Pf] \setminus C_{\twoess}$ and $F^{-1}$ to $C_{1}$, $C_{\twoess} \setminus C_{1}$, and $C[\Pf] \setminus C_{\twoess}$. 
(See \eqref{E:Cs.defn} and \eqref{E:F.on.parts}.) 
Trivially, by \eqref{E:F.on.C[Pf].defn}, $F$ is Lipschitz (w.r.t.\ $\xi_{+}$, \eqref{E:xi+.from.2.metrics}) on $C_{t^{-1}}$ with Lipschitz constant $t$ and, by \eqref{E:t:=.threeess/R}, on $C[\Pf] \setminus C_{\twoess}$ with Lipschitz constant $1 < t$. Similarly, $F^{-1}$ is Lipschitz on $C_{1} = F(C_{t^{-1}})$ with Lipschitz constant $1/t < 1$ and on and $C[\Pf] \setminus C_{\twoess}$ with Lipschitz constant $1 > t^{-1}$. (See \eqref{E:F.on.parts}.) We show $F$ and $F^{-1}$ are Lipschitz 
on the two sets $C_{\twoess} \setminus C_{t^{-1}}$ and 
$C_{\twoess} \setminus C_{1} = F(C_{\twoess} \setminus C_{t^{-1}})$, resp. 

Thus, by \eqref{E:F.F.invrs} and example \ref{Ex:ratnl.fns.loc.Lip}, we only need to prove that $(x,u) \mapsto \rho(x) |u|^{-1} u$ is Lipschitz in 
$(x,u) \in C_{\twoess} \setminus C_{t^{-1}} \supset C_{\twoess} \setminus C_{1} = F(C_{\twoess} \setminus C_{t^{-1}})$. 
Let $(x,u), (x',u') \in C_{\twoess} \setminus C_{t^{-1}}$. By definition \eqref{E:Cs.defn} 
of $C_{\twoess} \setminus C_{t^{-1}}$ and \eqref{E:rho.in.[0,1]}, we have  
    \begin{equation}  \label{E:|u|,|u'|,rho,t.ineqs}
      |u| < 2 \rho(x), \, |u'| < 2 \rho(x'), \, \rho(x) / |u| \leq t, 
        \, \text{ and } \rho(x') / |u'| \leq t,
    \end{equation}
Therefore,  
    \begin{multline}  \label{E:C.t.invrs.ineq} 
            \left| \frac{\rho(x)}{|u|} u - \frac{\rho(x')}{|u'|} u' \right| 
              \leq \left| \frac{\rho(x)}{|u|} u - \frac{\rho(x)}{|u|} u' \right| 
                + \left| \frac{\rho(x)}{|u|} u' - \frac{\rho(x')}{|u'|} u' \right| \\
                = \frac{\rho(x)}{|u|} | u - u' | + \left| \frac{\rho(x)}{|u|} 
                  - \frac{\rho(x')}{|u'|} \right| |u'|
                    \leq t |u - u'| + \left| \frac{\rho(x)}{|u|} - \frac{\rho(x')}{|u'|} \right| |u'|.
    \end{multline}
Therefore, we need only analyze $\left| \tfrac{\rho(x)}{|u|} - \tfrac{\rho(x')}{|u'|} \right| |u'|$. 
WLOG $|u| \geq |u'|$. By \eqref{E:|u|,|u'|,rho,t.ineqs}, 
$\rho(x')/|u| \leq \rho(x')/|u'| \leq t$. By \eqref{E:t:=.threeess/R}, $|u'|/|u| \leq 1 < 2 t$. 
By \eqref{E:rho.is.Lip.on.D}, $\rho$ is Lipschitz on $\D$ (w.r.t.\ $\xi$ and $| \cdot |$). 
    \begin{equation}   \label{E:K.is.Lip.const.for.rho}
        \text{ Let } K \in (1, \infty) \text{ be a Lipschitz constant for } 
          \rho \text{ w.r.t.\ } \xi.
    \end{equation}
 Therefore,
    \begin{align}
        \left| \frac{\rho(x)}{|u|} - \frac{\rho(x')}{|u'|} \right| |u'| 
          &\leq \frac{|u'|}{|u|} \bigl| \rho(x) - \rho(x') \bigr| 
              + \left| \frac{\rho(x')}{|u|} - \frac{\rho(x')}{|u'|} \right| |u'| \notag \\
            &\leq 2 t \bigl| \rho(x) - \rho(x') \bigr| 
              + \left| \frac{\rho(x')}{|u|} - \frac{\rho(x')}{|u'|} \right| |u'| \\
            &\leq 2 t K \xi(x, x') +  \frac{\rho(x')}{|u|} \bigl| |u'| - |u| \bigr|  \notag \\
            &\leq 2 t K \xi(x, x') +  t \bigl| |u'| - |u| \bigr|  \notag \\
            &\leq 2 t K \xi(x, x') +  t |u' - u|. \notag 
    \end{align}
Combining this with \eqref{E:C.t.invrs.ineq}, we have proof that $F$ and $F^{-1}$ are Lipschitz on $C_{\twoess} \setminus C_{t^{-1}}$. Specifically, in the case of $F$,
    \begin{align*}
        \xi_{+} \bigl[ F(x, v), F(x', v') \bigr] 
          &\leq \xi(x, x') + \left| \left( A v + B \frac{\rho(x)}{|v|} v \right) 
            - \left( A v' + B \frac{\rho(x')}{|v'|} v' \right) \right| \\
          &\leq \xi(x,x') + A |v-v'| + B \left| \frac{\rho(x)}{|v|} v - \frac{\rho(x')}{|v'|} v' \right| \\
          &\leq (1+2 B t K) \xi(x, x') + (A+2 Bt) |v - v'| \\
          &\leq (1+2 B t K) \xi(x, x') + 2(A+Bt) |v - v'| .
    \end{align*}
But from \eqref{E:A.B.in.trms.of.t}, it is easy to see that 
    \begin{equation} \label{E:A+Bt=t}
        A + B t = t. 
    \end{equation}
Moreover, by \eqref{E:t:=.threeess/R} and \eqref{E:R.<.threeess/3.A.B.ineqs}, 
$1+2 B t K < (1 + 2K) t$ and $(1 + 2K) t \geq 2$. Therefore, by \eqref{E:n.c.sqrd.sum.ineq} and \eqref{E:xi+.from.2.metrics}, 
    \begin{align*}
        \xi_{+} \bigl[ F(x, v), F(x', v') \bigr] &\leq (1+2 K) t \bigl( \xi(x, x') + |v - v'| \bigr) \\
          &\leq \sqrt{2} (1+2 K) t \, \xi_{+} \bigl[ (x, v), (x', v') \bigr] .
    \end{align*}
So, by \eqref{E:t:=.threeess/R}, the upper bound is inversely proportional to $R$. 

Thus, $F$ and $F^{-1}$ are Lipschitz on each of $C_{t^{-1}}$, 
$C_{\twoess} \setminus C_{t^{-1}}$, and $C[\Pf] \setminus C_{\twoess}$ (and their $F$-images, in the case of $F^{-1}$). And in all cases the Lipschitz constant of $F$ is proportional to $t$. This completes the proof of \eqref{E:F.and.F.invrs.Lip} in Case 0.

We show now that $F$ and $F^{-1}$ are each also Lipschitz \emph{across} $C_{t^{-1}}$, 
$C_{\twoess} \setminus C_{t^{-1}}$, and $C[\Pf] \setminus C_{\twoess}$, i.e., when their arguments are in different sets. Let $(x,u), (y,v) \in C[\Pf]$. Temporarily, redefining $p$, write 
    \begin{equation} \label{E:p.q.s.r.defns}
        p := \rho(x), \; q := |u|, \;  r := \rho(y), \text{ and } s := |v|. 
    \end{equation} 
Notice:
    \begin{equation} \label{E:|u-v|.geq.|q-s|}
       |u - v|^{2} = q^{2} - 2 u \cdot v + s^{2} \geq q^{2} - 2 qs + s^{2} 
         = |q - s|^{2} .
    \end{equation}

\emph{Case 1:} 
    \begin{equation}  \label{E:x,u,y,v.Case1}
      (x,u) \in C_{t^{-1}}, \; (y,v) \in C_{\twoess} \setminus C_{t^{-1}}.
    \end{equation} 
First, we prove $F$ is Lipschitz. Notice that, by \eqref{E:Cs.defn}, 
    \begin{equation}  \label{E:pqrs.Case1.ineqs}
      0 \leq q < p/t \text{ and } r/t \leq s < 2 r .
    \end{equation}
In particular, $r, s > 0$. By \eqref{E:xi+.from.2.metrics}, \eqref{E:F.on.C[Pf].defn}, and \eqref{E:t:=.threeess/R}, we have
	\begin{align} \label{E:Case1.ineq}
		\xi_{+} \bigl[ F(x,u), F(y,v) \bigr] 
		  &\leq \xi(x,y) + \bigl| t u - ( A + B r s^{-1} ) v \bigr| \\
		  &\leq \xi(x,y) + t |u-v| + \Bigl| t - \bigl( A + B r s^{-1} \bigr) \Bigr| s . \notag\\
		  &\leq t \xi(x,y) + t |u-v| + \Bigl| t - \bigl( A + B r s^{-1} \bigr) \Bigr| s . \notag
	\end{align}
Suppose $u = 0$. Now, by \eqref{E:pqrs.Case1.ineqs}, $s^{-1`} \leq t/r$, so by \eqref{E:A+Bt=t},
    \begin{equation}  \label{E:Case1.u=0}
      \bigl| t u - ( A + B r s^{-1} ) v \bigr| = 
        \bigl( A + B r s^{-1} \bigr) |u-v| \leq \bigl( A + B t \bigr) |u-v| = t |u-v| .
    \end{equation}
So we may assume $q > 0$.

We show that 
	\begin{equation}  \label{E:t.geq.A+Br s.invrs.Case.1}
		t \geq A + B r s^{-1} \geq 1 \text{ and } 
		  A + B p q^{-1} > t .
	\end{equation}
By \eqref{E:pqrs.Case1.ineqs}, $A + B r s^{-1} \leq A + B t = t$ by \eqref{E:A+Bt=t}. 
Moreover, $r s^{-1} > 1/2$. Therefore, by \eqref{E:AB.dilate.defn},
    \begin{equation*}
        A + B r s^{-1} \geq A + 2 (1 - A) / 2 = 1.
    \end{equation*}
By \eqref{E:pqrs.Case1.ineqs}, $p q^{-1} > t$. Hence, $A + Bp q^{-1} > t$ by \eqref{E:A+Bt=t} again. 
\eqref{E:t.geq.A+Br s.invrs.Case.1} follows. Therefore, we get to drop an absolute value sign in \eqref{E:Case1.ineq}:
	\begin{equation} \label{E:xi+.FF.no.abs.sign}
		\xi_{+} \bigl[ F(x,u), F(y,v) \bigr] 
		\leq t \xi(x,y) + t |u-v| + \Bigl[ t - \bigl( A + B r s^{-1} \bigr) \Bigr] s.
	\end{equation}

Next, we show that,  
	\begin{equation}  \label{E:conjectured.t.diff.bound}
		\Bigl[ t - \bigl( A + B r s^{-1} \bigr) \Bigr] s \leq 2 t \bigl( s - r/t \bigr) .
	\end{equation}
By \eqref{E:pqrs.Case1.ineqs} again, the RHS of \eqref{E:conjectured.t.diff.bound} is non-negative and $r s^{-1} \leq t$. By \eqref{E:A+Bt=t}, \eqref{E:conjectured.t.diff.bound} holds 
if $r/t = s$, which means $r s^{-1} = t$. Assume $r/t < s$, which means $r s^{-1} < t$. We have, by \eqref{E:A.B.in.trms.of.t} and clearing fractions, 
	\begin{align*}  
	\frac{t - \bigl( A + B r/s \bigr)}{s - r/t} 
	   &= \frac{ 2t^{2} - t - \bigl[ t + 2(t-1)r/s \bigr] }{(2t-1)(s-r/t)} \\
	   &= 2 \frac{ s t^{2} - s t - r t + r }{s(2t-1)(s-r/t)} \\
	   &= 2 \frac{ s t^{3} - s t^{2} - r t^{2} + r t}{s(2t-1)(st-r)} \\
	   &= 2 \frac{ t (st - r)(t - 1)}{s(2t-1)(st-r)} \\
	   &= 2 \frac{ t (t - 1)}{s(2t-1)} .
	\end{align*}
Therefore, we have
	\begin{equation} \label{E:t-(A+Br/s).ratio}
		  \frac{\bigl[ t - ( A + B r s^{-1} ) \bigr] s}{s - r/t} = 2 \frac{ t (t - 1)}{2t-1} \leq 2 t.
	\end{equation}
This proves \eqref{E:conjectured.t.diff.bound}.

Now, by \eqref{E:pqrs.Case1.ineqs} and \eqref{E:p.q.s.r.defns}, $|u| - \rho(x)/t = q - p/t < 0$. Hence, by \eqref{E:p.q.s.r.defns}, \eqref{E:t:=.threeess/R}, and \eqref{E:K.is.Lip.const.for.rho},  
	\begin{align}  \label{E:t(s-r/t).Case1.ineq}
         2t (s - r/t) &= 2t \Bigl[ \bigl(|v| - |u| \bigr) + \bigl( |u| - \rho(x)/t \bigr) 
	    + \bigl( \rho(x) - \rho(y) \bigr)/t \Bigr] \notag \\
	    &\leq 2t \Bigl[ |v - u| + \bigl( \rho(x) - \rho(y) \bigr)/t \Bigr] \\
	    &\leq 2t \Bigl[ |v - u| + K \xi(x,y)/t \Bigr]  \notag \\
	    &\leq 2t \Bigl[ |v - u| + K \xi(x,y) \Bigr] \notag .
	\end{align}
Combining this with \eqref{E:xi+.FF.no.abs.sign} and \eqref{E:conjectured.t.diff.bound},   Lipschitz-osity of $F$ (with Lipschitz constant proportional to $t$) is proved in Case 1. 

Now $F^{-1}$ in Case 1. Let $L \in (1, \infty)$. We specify it presently. Recall \eqref{E:F.on.C[Pf].defn}, \eqref{E:x,u,y,v.Case1}, \eqref{E:p.q.s.r.defns}, and \eqref{E:t:=.threeess/R}. By \eqref{E:pqrs.Case1.ineqs}, $|v| = s \geq r/t$. Hence, by \eqref{E:xi+.from.2.metrics},  \eqref{E:n.c.sqrd.sum.ineq}, \eqref{E:t.geq.A+Br s.invrs.Case.1}, and \eqref{E:t-(A+Br/s).ratio}, 
	\begin{align*}
		\sqrt{2} \, L \, \xi_{+} \bigl[ F(x,u), & F(y,v) \bigr] 
		  \geq L \, \xi(x,y) + L\bigl| t u - ( A + B r s^{-1} ) v \bigr| \\
		  &\geq L \, \xi(x,y) + \bigl| t u - ( A + B r s^{-1} ) v \bigr| \\
		  &= L \, \xi(x,y) + \bigl| t (u-v) - ( A + B r s^{-1} -t ) v \bigr| \\
		  &\geq L \, \xi(x,y) + t |u-v| - \bigl| ( A + B r s^{-1} ) - t \bigr| s  \\
		  &= L \, \xi(x,y) + t |u-v| - \bigl[ t - ( A + B r s^{-1} ) \bigr] s \\
		  &= L \, \xi(x,y) + t |u-v| - 2 \frac{ t (t - 1)}{2t-1} (s - r/t) \\
		  &= L \, \xi(x,y) + t |u-v| \\
		  & \qquad - 2 \frac{ t (t - 1)}{2t-1} 
		    \Bigl[ \bigl( |v| - |u| \bigr) + \bigl( |u| - \rho(x)/t \bigr) 
		      + \bigl( \rho(x) - \rho(y) \bigr)/t \Bigr] .
        \end{align*}
As in \eqref{E:t(s-r/t).Case1.ineq}: By \eqref{E:pqrs.Case1.ineqs} and \eqref{E:p.q.s.r.defns}, 
$|u| - \rho(x)/t = q - p/t < 0$. Applying \eqref{E:K.is.Lip.const.for.rho} we get, 
	\begin{align*}
		\sqrt{2} \, L \, \xi_{+} \bigl[ F(x,u), & F(y,v) \bigr] \\
		  &\geq L \, \xi(x,y) + t |u-v|  \\
		  & \qquad - 2 \frac{ t (t - 1)}{2t-1} \Bigl[ \bigl( |v| - |u| \bigr) \bigr) 
		      + \bigl( \rho(x) - \rho(y) \bigr)/t \Bigr] \\
		  &\geq L \, \xi(x,y) + t |u-v| - 2 \frac{ t (t - 1)}{2t-1} 
		    \Bigl[ \bigl( |v| - |u| \bigr) \bigr) 
		      + K \xi(x,y)/t \Bigr] \\
		  &\geq \left( L -  2 \frac{ t - 1}{2t-1} K \right) \xi(x,y) + t |u-v| -
		     2 \frac{ t (t - 1)}{2t-1} |u-v| \\
		  &= \left( L -  2 \frac{ t - 1}{2t-1} K \right) \xi(x,y) 
		    + \left( t - 2 \frac{ t (t - 1)}{2t-1} \right) |u-v| \\
		  &= \left( L -  2 \frac{ t - 1}{2t-1} K \right) \xi(x,y) + \frac{t}{2t-1} |u-v|. 
	\end{align*}
By \eqref{E:t:=.threeess/R}, we have $ 2 (t - 1)/(2t-1) > 1/2$ and $t/(2t-1) > 1/2$. By \eqref{E:K.is.Lip.const.for.rho}, $K > 1$. Let
    \begin{equation*}
      L := 2 \frac{ t - 1}{2t-1} K + \frac{1}{2} > 1, 
    \end{equation*}
Then, by \eqref{E:xi+.from.2.metrics},
    \begin{equation*}
      2 \sqrt{2} \, L \, \xi_{+} \bigl[ F(x,u), F(y,v) \bigr] \geq \xi(x,y) + |u-v|
        \geq \xi_{+} \bigl( (x,u), (y,v) \bigr) .
    \end{equation*}
Replace $(x,u)$ by $F^{-1}(x,w)$ and $(y,v)$ by $F^{-1}(y,z)$. Then we see that $F^{-1}$, is Lipschitz in Case 1. Combining this with what we proved about $F$, this completes the proof of \eqref{E:F.and.F.invrs.Lip} in Case 1. 

\emph{Case 2:} $(x,u) \in C_{\twoess} \setminus C_{t^{-1}}$ 
and $(y,v) \in C[\Pf] \setminus C_{\twoess}$. Recall \eqref{E:p.q.s.r.defns} and 
\eqref{E:t:=.threeess/R}. Thus, by \eqref{E:Cs.defn}, 
	\begin{equation}  \label{E:Case2.pqrst.ineqs}
		p/t \leq q < 2p, \text{ so } 1/2 < p q^{-1} \leq t,
		   \text{ and } s \geq 2 r. 
	\end{equation}

First, we prove $F$ is Lipschitz. By \eqref{E:F.on.C[Pf].defn}, we have
	\begin{multline}  \label{E:xi+.r.qinvrs.dist.in.Case2}
		\xi_{+} \bigl[ F(x,u), F(y,v) \bigr] 
		  \leq \xi(x,y) + \bigl| ( A + B p q^{-1} ) u - v \bigr| \\
		  \leq \xi(x,y) + (A + B p q^{-1}) |u-v| + \bigl| ( A + B p q^{-1} ) - 1 \bigr| s. 
	\end{multline}
By \eqref{E:Case2.pqrst.ineqs}, \eqref{E:A+Bt=t}, and \eqref{E:R.<.threeess/3.A.B.ineqs}, 
	\begin{equation}  \label{E:t.geq.A+Bpq.invrs.Case.2}
		  t \geq A + B p q^{-1} \geq 1 .
	\end{equation}
Hence, we get to drop the last $| \cdot |$ sign in \eqref{E:xi+.r.qinvrs.dist.in.Case2} to get, after applying \eqref{E:Case2.pqrst.ineqs} and \eqref{E:A+Bt=t} again,
	\begin{multline}   \label{E:xi+.dist.in.Case2}
		\xi_{+} \bigl[ F(x,u), F(y,v) \bigr] 
		      \leq \xi(x,y) + (A + B t) |u-v| + \bigl[ ( A + B p q^{-1} ) - 1 \bigr] s \\
			  \leq \xi(x,y) + t |u-v| + \bigl[ ( A + B p q^{-1} ) - 1 \bigr] s .
	\end{multline}

Notice that, by \eqref{E:A.B.in.trms.of.t} and \eqref{E:Case2.pqrst.ineqs}, 
	\begin{align}  \label{E:A+Br/q.Case2}
		 ( A + B p q^{-1} ) - 1 
    		  &= \frac{t}{2t-1} + \frac{2t-2}{2t-1} \left( \frac{p}{q} \right) - 1\notag \\
    		  &= \frac{ 2tp - 2p- tq + q }{(2t-1)q} \\
    		  &= \frac{(2p - q)(t - 1)}{(2t-1)q} . \notag 
	\end{align}
But, by \eqref{E:Case2.pqrst.ineqs} and \eqref{E:|u-v|.geq.|q-s|}, 
    \begin{multline}  \label{E:2p-q.ineq}
      2p -q = (2p-2r) + (2r - q) \leq 2(p-r) + (s-q) \\
        \leq 2K\xi(x,y) + |s-q| \leq 2K\xi(x,y) + |u-v| .
    \end{multline}

\emph{Assume} first that $s \leq q$. (See \eqref{E:p.q.s.r.defns}. But if 
$(x,u) \in C_{\twoess} \setminus C_{t^{-1}}$ 
and $(y,v) \in C[\Pf] \setminus C_{\twoess}$ how can $s \leq q$? It can happen if $r < p$. See \eqref{E:Case2.pqrst.ineqs}.) Then from, \eqref{E:xi+.dist.in.Case2}, 
\eqref{E:A+Br/q.Case2}, \eqref{E:xi+.from.2.metrics}, the preceding, \eqref{E:t:=.threeess/R}, and \eqref{E:n.c.sqrd.sum.ineq},
    \begin{align}  \label{E:Case2,s.leq.q}
      \xi_{+} \bigl[ F(x,u), F(y,v) \notag \bigr] 
        &\leq \xi(x,y) + t |u-v| + \frac{(2p - q)(t - 1)}{2t-1} \frac{s}{q} \notag \\
        &\leq \xi(x,y) + t |u-v| + \frac{(2p - q)(t - 1)}{2t-1} \notag \\
        &\leq \xi(x,y) + t |u-v| + \bigl( 2K\xi(x,y) + |u-v| \bigr) \frac{t - 1}{2t-1} \\
        &\leq \xi(x,y) + t |u-v| + \bigl( 2K\xi(x,y) + |u-v| \bigr) \notag \\
        &\leq 3t \xi(x,y) + 3 t|u-v| \bigr) \notag \\
        &\leq 3t \sqrt{2} \, \xi_{+} \bigl[ (x,u), y,v) \bigr] \notag ,
    \end{align}
if $s \leq q$. By \eqref{E:t:=.threeess/R}, this bound is inversely proportional to $R$.
 
Next, assume $s > q$. We have
  \begin{align}  \label{E:|A.B.u-v|.beta.brkdwn.1}
     \bigl| ( A + B p q^{-1} ) u - v \bigr|
       &\leq ( A + B p q^{-1}) \bigl| u -  q s^{-1} v \bigr|
         + \bigl| ( A + B p q^{-1} ) q s^{-1} v - v \bigr| \notag \\
       &= ( A + B p q^{-1}) \bigl| u -  q s^{-1} v \bigr| 
         + \bigl| ( A + B p q^{-1} ) q s^{-1} - 1  \bigr| s \\
       &= ( A + B p q^{-1}) \bigl| u -  q s^{-1} v \bigr| 
         + \bigl| ( A + B p q^{-1} ) q - s \bigr| . \notag 
  \end{align}
Let $\beta := q s^{-1}$, so $1 - \beta > 0$. Then  
    \begin{multline}  \label{E:|u-v|.beta.brkdwn}
      | u - v |^{2} = \bigl| (u - \beta v) - (1 - \beta) v \bigr|^{2} \\ 
        = \bigl| u -  \beta v \bigr|^{2} - 2 (1 - \beta) (u - \beta v) \cdot v 
          + (1 - \beta)^{2} s^{2} .
    \end{multline}
Now, $q = \beta s$, so
    \begin{align*}
       - 2 (1 - \beta) (u -  &\beta v) \cdot v + (1 - \beta)^{2} s^{2} \\
         &= - 2 (1 - \beta) u \cdot v + 2 (1 - \beta) \beta s^{2} + (1 - \beta)^{2} s^{2}  \\
         &\geq - 2 (1 - \beta) q s + 2 (1 - \beta) \beta s^{2} + (1 - \beta)^{2} s^{2} \\
         &= - 2 (1 - \beta) \beta s^{2} + 2 (1 - \beta) \beta s^{2} + (1 - \beta)^{2} s^{2} \\
         &= (1 - \beta)^{2} s^{2} \geq 0 .
    \end{align*}
Substituting this into \eqref{E:|u-v|.beta.brkdwn}, we get 
    \begin{equation}  \label{E:s.geq.q.|u-v|.geq.|u-q.invrs.v|}
      \text{If } s \geq q \text{ then } |u-v| \geq \bigl| u -  q s^{-1} v \bigr| .
    \end{equation} 
\emph{Note that this does not depend on Case 2 assumptions. It is always true.} Moreover, by \eqref{E:Case2.pqrst.ineqs}, $A + Bp q^{-1} \leq A + B t = t$, by \eqref{E:A+Bt=t}. Applying these two facts to \eqref{E:|A.B.u-v|.beta.brkdwn.1}, we get 
  \begin{equation}  \label{E:|A.B.u-v|.beta.brkdwn.2}
     \bigl| ( A + B p q^{-1} ) u - v \bigr| 
       \leq t | u -  v | + \bigl| ( A + B p q^{-1} ) q - s \bigr| .
  \end{equation}

Suppose $( A + B p q^{-1} ) q \geq s$. By \eqref{E:Case2.pqrst.ineqs}, $s \geq 2 r$ and, by \eqref{E:R.<.threeess/3.A.B.ineqs}, $2 A r + Br = 2r$, so $s \geq 2 A r + Br$. This means
    \begin{equation*}
      \bigl| ( A + B p q^{-1} ) q - s \bigr| = ( A + B p q^{-1} ) q - s
        \leq ( Aq + B p) - (2 A r + Br) = A(q - 2r) + B(p-r) .
    \end{equation*}
Now, $q-2p < 0$, by \eqref{E:Case2.pqrst.ineqs}. Hence, 
    \begin{equation*}
      A(q - 2r) = A \bigl[ (q-2p) + 2(p-r) \bigr] \leq 2A(p-r)    
        \leq 2A K \xi(x,y) . 
    \end{equation*}
Therefore, if $( A + B p q^{-1} ) q \geq s$, we have, by \eqref{E:R.<.threeess/3.A.B.ineqs}, 
    \begin{equation} \label{E:A+B p.qinvrs.q-s.ineq}
      \bigl| ( A + B p q^{-1} ) q - s \bigr| 
        \leq A(q - 2r) + B(p-r) \leq (2A+B) K \xi(x,y) = 2 K \xi(x,y) .
    \end{equation}
Substituting this into \eqref{E:|A.B.u-v|.beta.brkdwn.2} and applying \eqref{E:n.c.sqrd.sum.ineq}, we get
    \begin{multline*}
      \xi_{+} \bigl[ F(x,u), F(y,v) \bigr] \leq \xi(x,y) +  \bigl| ( A + B p q^{-1} ) u - v \bigr| \\
        \leq (1+2K) \xi(x,y) + t | u -  v |) 
          \leq \sqrt{2} \, t (1+2K) \xi_{+} \bigl[ (x,u), (y,v) \bigr] ,
    \end{multline*}
a bound inversely proportional to $R$, by \eqref{E:t:=.threeess/R}.

Next, suppose $( A + B p q^{-1} ) q \leq s$. By \eqref{E:F.F.invrs} and \eqref{E:f.dil.is.dilation},  
$( A + B p q^{-1} ) q \geq q$. Thus, 
    \begin{equation*}
      \bigl| ( A + B p q^{-1} ) q - s \bigr| = s - ( A + B p q^{-1} ) q \leq s - q \leq |u-v|,
    \end{equation*}
by \eqref{E:|u-v|.geq.|q-s|}. Combining this with \eqref{E:|A.B.u-v|.beta.brkdwn.2},  \eqref{E:xi+.r.qinvrs.dist.in.Case2}, and \eqref{E:t:=.threeess/R}, we have 
    \begin{align}  \label{E:F.Lip.in.Case.2}
      \xi_{+} \bigl[ F(x,u), F(y,v) \bigr] &\leq \xi(x,y) + \bigl| ( A + B p q^{-1} ) u - v \bigr| 
        \notag \\
        &\leq t \xi(x,y) + (t+1) \bigl| u - v \bigr| \\ 
        &\leq 2 t \xi(x,y) + 2 t \bigl| u - v \bigr| \notag \\
        &\leq 2 \sqrt{2} \, t \xi_{+} \bigl[ (x,u), (y,v) \bigr] \notag
    \end{align}
if $s > q$. This is inversely proportional to $R$. Combine this with \eqref{E:Case2,s.leq.q}, we have, in Case 2, that $F$ Lipschitz with Lipschitz constant inversely proportional to $R$. 

Now we consider $F^{-1}$ in Case 2. Let $L \geq 6K/5 + 2/5 > 1$, by \eqref{E:K.is.Lip.const.for.rho}. We have, by \eqref{E:n.c.sqrd.sum.ineq} and \eqref{E:t.geq.A+Bpq.invrs.Case.2}, 
	\begin{align}  \label{E:Case2.inverse}
		\sqrt{2} \, L \xi_{+} \bigl[ F(x,u), F(y,v) \bigr] 
		  &\geq L \xi(x,y) + \bigl| ( A + B p q^{-1} ) u - v \bigr| \notag \\
		  &= L \xi(x,y) + \Bigl| \bigl[ ( A + B p q^{-1} ) u - u \bigr] - (v-u) \Bigr|  \\
		  &\geq L \xi(x,y) + |u-v| - \bigl| ( A + B p q^{-1} ) - 1 \bigr| q  \notag \\
		  &\geq L \xi(x,y) + |u-v| - \bigl( ( A + B p q^{-1} ) - 1 \bigr) q . \notag 
	\end{align}
By \eqref{E:t:=.threeess/R}, $(t - 1)/(2t-1) \leq 3/5$. Multiplying both sides 
of \eqref{E:A+Br/q.Case2} by $q$ and using \eqref{E:2p-q.ineq},  
we see,
	\begin{equation*}
		\bigl( ( A + B r q^{-1} ) - 1 \bigr) q = \frac{(2p - q)(t - 1)}{2t-1} 
		  \leq \frac{t - 1}{2t-1} \bigl( 2K\xi(x,y) + |u-v| ) 
		    \leq \frac{3}{5} \bigl( 2K\xi(x,y) + |u-v| ) .
	\end{equation*}

Hence, by \eqref{E:Case2.inverse},  
	\begin{align*} 
		\sqrt{2} \, L \xi_{+} \bigl[ F(x,u), F(y,v) \bigr] 
		  &\geq L \xi(x,y) + |u-v| - \frac{3}{5} \bigl( 2K\xi(x,y) + |u-v| ) \notag \\
		  &\geq \left( L  - \frac{6}{5} K \right) \xi(x,y) + \frac{2}{5} |u-v|  \\
		  &\geq \frac{2}{5} \bigl( \xi(x,y) + |u-v| \bigr)  \notag \\
		  &\geq \frac{2}{5} \xi_{+} \bigl[ (x,u), (y,v) \bigr]  \notag .
	\end{align*}
This proves $F^{-1}$ is Lipschitz in the Case 2 situation. Combining this with \eqref{E:F.Lip.in.Case.2}, we have that \eqref{E:F.and.F.invrs.Lip} holds in Case 2.

\emph{Case 3:} $(x,u) \in C_{t^{-1}}$ and $(y,v) \in C[\Pf] \setminus C_{\twoess}$. First, we prove $F$ is Lipschitz. Recall \eqref{E:p.q.s.r.defns}. As in Case 1 (see \eqref{E:Case1.u=0}) we may assume $q > 0$.
  	\begin{equation}  \label{E:pqr.ineqs.Case3}
		0 < q < t^{-1} p \text{ so } p q^{-1} > t, \, F(x,u) = (x, t u),
		   \text{ and } s \geq 2r.
	\end{equation}
Therefore, by \eqref{E:F.from.f.dil}, \eqref{E:F.on.C[Pf].defn}, and \eqref{E:Cs.defn}, 
we have
	\begin{equation}  \label{E:xi+.r.qinvrs.dist.in.Case3}
		\xi_{+} \bigl[ F(x,u), F(y,v) \bigr] 
		  \leq \xi(x,y) + \bigl| t u - v \bigr| 
		  \leq \xi(x,y) + t |u-v| + (t - 1) s. 
	\end{equation}
But, by \eqref{E:pqr.ineqs.Case3}, $t-1 \leq p/q - 1$ and $r - s \leq 0$ so, 
by \eqref{E:|u-v|.geq.|q-s|} and \eqref{E:K.is.Lip.const.for.rho}, 
    \begin{multline}   \label{E:t-1.ineq}
        t-1 \leq p/q - 1= (p-q)/q \\ 
          = \big[ (p-r) + (r-s) + (s-q) \bigr] / q 
            \leq \big[ K \xi(x,y) + |u-v| \bigr] / q .
    \end{multline}
Suppose $s \leq q$. (Possible if $r < t^{-1} p/2 < p/6$. See \eqref{E:pqr.ineqs.Case3} and \eqref{E:t:=.threeess/R}. ) Then substituting the preceding into \eqref{E:xi+.r.qinvrs.dist.in.Case3} we get, by \eqref{E:K.is.Lip.const.for.rho}, \eqref{E:n.c.sqrd.sum.ineq}, and \eqref{E:t:=.threeess/R},
    \begin{multline}  \label{E:Lip.bound.s.leq.q.Case3}
      \xi_{+} \bigl[ F(x,u), F(y,v) \bigr] \leq (K+1) \xi(x,y) + 2 t |u-v| \\
        \leq (K+1) t \bigl( \xi(x,y) + |u-v| \bigr)
       \leq 2 t \sqrt{2} \, K \xi_{+} \bigl[ (x,u), (y,v) \bigr] .
    \end{multline}
By \eqref{E:t:=.threeess/R}, this is inversely proportional to $R$.

So we may, until further notice assume $s > q$. We have
  \begin{align}  \label{E:|tu-v|.q.s.invrs}
     \bigl| t u - v \bigr|
       &\leq t \bigl| u -  q s^{-1} v \bigr|
         + \bigl| t q s^{-1} v - v \bigr| \notag \\
       &= t \bigl| u -  q s^{-1} v \bigr| 
         + \bigl| t q s^{-1} - 1  \bigr| s \\
       &= t \bigl| u -  q s^{-1} v \bigr| 
         + \bigl| t q - s \bigr| . \notag 
  \end{align}
Thus, from \eqref{E:|tu-v|.q.s.invrs} and \eqref{E:s.geq.q.|u-v|.geq.|u-q.invrs.v|}, 
  \begin{equation}  \label{E:|tu-v|.tq}    
     \bigl| t u - v \bigr| 
       \leq t | u -  v | + \bigl| t q - s \bigr| .
  \end{equation}

Assume $t q \geq s$. By \eqref{E:pqr.ineqs.Case3}, $tq < p$ and $s \geq 2 r$. The assumption $s > q$ remains in force. Therefore, 
    \begin{equation}  \label{E:Eq-s.leq.2K.xi(x,y)}
      \bigl| t q - s \bigr| = t q - s <  p - s \leq p - 2 r \leq 2(p - r) \leq 2 K \xi(x,y) .
    \end{equation}
Hence, by \eqref{E:|tu-v|.q.s.invrs}, we have
  \begin{equation*}  
     \bigl| t u - v \bigr| 
       \leq t | u -  v | + 2 K \xi(x,y) .
  \end{equation*}
Therefore, by \eqref{E:xi+.r.qinvrs.dist.in.Case3} and since $t \geq 3 > 1$, 
   \begin{align}  \label{E:Lip.bound.t.q.geq.s.Case.3} 
      \xi_{+} \bigl[ F(x,u), F(y,v) \bigr] 
        &\leq \xi(x,y) + \bigl| t u - v \bigr| \notag \\
        &\leq (2K+1) \xi(x,y) + t |u - v| \\
        &\leq \sqrt{2} \, (2K+t) \xi_{+} \bigl[ (x,u), (y,v) \bigr] \notag \\
        &\leq t \sqrt{2} \, (2K+1) \xi_{+} \bigl[ (x,u), (y,v) \bigr] . \notag 
    \end{align}
By \eqref{E:t:=.threeess/R}, this is inversely proportional to $R$. 

Next, assume $t q \leq s$. Then, by \eqref{E:t:=.threeess/R} and \eqref{E:|u-v|.geq.|q-s|}, 
    \begin{equation*}
      \bigl| t q - s \bigr| = s - t q \leq s - q = |s-q| \leq |u-v| .
    \end{equation*}
Combining this with \eqref{E:|tu-v|.tq} we have, 
    \begin{equation*}
      \bigl| t u - v \bigr| \leq (t+1)|u-v| .
    \end{equation*} 
By \eqref{E:xi+.r.qinvrs.dist.in.Case3} and using \eqref{E:t:=.threeess/R}, we have 
    \begin{align*}
      \xi_{+} \bigl[ F(x,u), F(y,v) \bigr] 
        &\leq \xi(x,y) + \bigl| t u - v \bigr| \\
        &\leq \xi(x,y) + (t+1)|u - v| \\
        &\leq 2 t \sqrt{2} \, \xi_{+} \bigl[ (x,u), (y,v) \bigr] . 
    \end{align*}
Again, inversely proportional to $R$. This together with \eqref{E:Lip.bound.s.leq.q.Case3} and \eqref{E:Lip.bound.t.q.geq.s.Case.3}, shows that $F$ is Lipschitz with constant inversely proportional to $R$ in Case 3.

As a check, here are the cases we have considered under Case 3, $F$:
    \begin{itemize}
    \item $q = 0$,
    \item $q > 0$:
        \begin{enumerate}
            \item $s \leq q$,
            \item $s > q$:
                \begin{enumerate}
                \item $t q \geq s$,
                \item $t q \leq s$.
                \end{enumerate}
        \end{enumerate}
    \end{itemize}

Now we consider $F^{-1}$ in Case 3. If $q := |u| = 0$, then $| t u - v | = | u - v |$. If $s := |v| = 0$, then $| t u - v | = t | u - v |$. In either case, by \eqref{E:t:=.threeess/R}, 
    \begin{equation*}
      \sqrt{2} \,  \xi_{+} \bigl[ F(x,u), F(y,v) \bigr]  \geq \xi(x,y) + | t u - v | 
       \geq \xi(x,y) + | u - v | \geq \xi_{+} \bigl[ (x,u), (y,v) \bigr] 
    \end{equation*}
and we are done. So assume $q, s > 0$, so, by \eqref{E:pqr.ineqs.Case3}, $p > 0$.
 
Notice that, by \eqref{E:n.c.sqrd.sum.ineq}, 
$\xi_{+} \bigl[ (x,u), (y,v) \bigr] \leq \sqrt{2} \bigl[ \xi(x,y) + |u-v| \bigr]$ and 
$\xi_{+} \bigl[ F(x,u), F(y,v) \bigr] \geq (1/\sqrt{2}) \bigl[ \xi(x,y) + |tu - v| \bigr]$. If $x=y$ 
then $\xi(x,y) = 0$. On the other hand, if $\xi(x,y) = 0$ then $p = r$. (See \eqref{E:p.q.s.r.defns}.) Thus, by \eqref{E:pqr.ineqs.Case3}, $0 < tq < p = r < 2r \leq s$. 
Therefore, $|tu - v| \geq s - tq \geq 2r - r = r > 0$. We conclude that $\xi(x,y) + |tu - v| > 0$ whether $x=y$ or not. Thus, we may take the quotient and get 
    \begin{equation}  \label{E:xi+.ratio.ineq}
      \frac{\xi_{+} \bigl[ (x,u), (y,v) \bigr]}{\xi_{+} \bigl[ F(x,u), F(y,v) \bigr] }
        \leq 2 \frac{\xi(x,y) + |u-v|}{\xi(x,y) + |tu - v|}
          \leq 2 + 2 \frac{|u-v|}{\xi(x,y) + |tu - v|} .
    \end{equation}
If we bound this by a finite constant, we are done. Let 
    \begin{equation}  \label{E:Case3.X.defn}
      X := \frac{|u-v|}{\xi(x,y) + |tu - v|} .
    \end{equation}
Call an inequality that bounds $X$ above by a finite constant (constant in appropriate $u$ and $v$) a ``Lipschitz inequality''. If we can derive such an inequality then we will have proved that $F^{-1}$ is Lipschitz in Case 3. Observe that if $u=v$, then 
$X = 0$ and $1$ is an upper bound. So assume $u \neq v$. 
 
Let $\gamma$ denote the cosine of the angle between $u$ and $v$. 
i.e., $\gamma := (u \cdot v)/(qs)$. (Here, ``$\gamma$'' means something different from what it means in \eqref{E:Hm.a.S.geq.R.d-p-1}.) We have
    \begin{equation*}
      | t u - v | ^{2} - | u - v | ^{2} = (t^{2} - 1) q^{2} - 2(t-1) \gamma qs .
    \end{equation*}
This is positive if $\gamma \leq 0$, in which case $|tu - v| \geq |u-v|$ and, with $u \neq v$, we again obtain a Lipschitz inequality: $X \leq 1$. So we may certainly assume 
    \begin{equation} \label{E:gamma.in.[0,1]}
      \gamma \in [0,1] .
    \end{equation}
    
 Next, with $q$ and $s$ held fixed consider the function
    \begin{equation} \label{E:H(gamma).defn}
       X^{2} \leq H(\gamma) := \frac{|u-v|^{2}}{|tu-v|^{2}}
        = \frac{q^{2} - 2 \gamma q s + s^{2}}{t^{2} q^{2} - 2 t \gamma q s + s^{2}} ,
          \qquad 0 \leq \gamma \leq 1 .
    \end{equation}
We have 
    \begin{equation*}
      H'(\gamma) 
        = \frac{ -2qs ( t^{2} q^{2} - 2 t \gamma q s + s^{2} ) + 2tqs(q^{2} - 2 \gamma q s 
          + s^{2}) }{(t^{2} q^{2} - 2 t \gamma q s + s^{2})^{2} } .
    \end{equation*}
We have 
    \begin{equation}  \label{E:Case3.num.of.H'}
      \text{Numerator of } H'(\gamma) = 2sq(t-1)(s^{2} - t q^{2}) . 
    \end{equation}
Thus, by \eqref{E:t:=.threeess/R}, if $\sqrt{t} \, q \geq s$, $H$ is non-increasing. Moreover, 
if $\sqrt{t} \, q \geq s$, by \eqref{E:t:=.threeess/R} again, 
     \begin{equation*}
       X^{2} \leq H(\gamma) \leq H(0)
        = \frac{q^{2} + s^{2}}{t^{2} q^{2} + s^{2}} \leq 1, 
          \qquad 0 \leq \gamma \leq 1 .
    \end{equation*}
Which immediately gives a Lipschitz inequality. 

So suppose $s \geq \sqrt{t} \, q$. Then, by \eqref{E:t:=.threeess/R} yet again, 
    \begin{equation} \label{E:s.geq.q}
      s \geq \sqrt{t} \, q \geq q .
    \end{equation}

Suppose $q \leq s \leq tq$. 
By \eqref{E:pqr.ineqs.Case3} this means
$2r \leq s \leq tq \leq p$. I.e., $r \leq p/2$ so $p-r \geq p/2$. Moreover, 
by \eqref{E:gamma.in.[0,1]}, 
    \begin{equation}  \label{E:q<s<ts.{u-v}.ineq}
      |u - v|^{2} = q^{2} - 2 \gamma qs + s^{2} \leq q^{2} + s^{2} \leq q^{2} (1 + t^{2}) 
        \leq \frac{1 + t^{2}}{t^{2}} p^{2} .
    \end{equation}
Now, by \eqref{E:K.is.Lip.const.for.rho}, $K \xi(x,y) \geq p - r \geq p/2$. Therefore, 
by \eqref{E:q<s<ts.{u-v}.ineq},
    \begin{multline*}
      X = \frac{|u-v|}{\xi(x,y) + |tu - v|} 
        \leq  \frac{|u-v|}{\xi(x,y)}
          \leq K \sqrt{ \frac{1 + t^{2}}{t^{2}} } \frac{p}{K \xi(x,y)} \\
            \leq K \sqrt{ \frac{1 + t^{2}}{t^{2}} } \frac{p}{p/2} 
              = 2 K \sqrt{ \frac{1 + t^{2}}{t^{2}} } < \infty ,
    \end{multline*}
a Lipschitz inequality.

Next, suppose $s \geq tq + p/3$. We are still assuming $s \geq \sqrt{t} \, q$ so, by \eqref{E:Case3.num.of.H'}, $H'(\gamma) \geq 0$. Therefore, 
$H(\gamma)$ as defined in \eqref{E:H(gamma).defn} is bounded above by $H(1)$. Thus,
    \begin{equation}  \label{E:|u-v|/|t u - v|.qst}
      \frac{|u-v|}{|t u - v|} = H(\gamma) \leq H(1) 
        = \frac{s-q}{s - tq} = \frac{1-q/s}{1 - tq/s} .
    \end{equation}
By supposition, 
    \begin{equation*}  \label{E:Case3.tqs.frac.ineq}
      \frac{tq}{s} \leq \frac{tq}{tq + p/3} .
    \end{equation*}
But $tq/(tq + p/3)$ is increasing in $tq$. Moreover, $0 < tq < p$, by \eqref{E:pqr.ineqs.Case3}. Hence,
    \begin{equation*}
        \frac{tq}{s} \leq \frac{tq}{tq + p/3} \leq \frac{p}{p + p/3} = \frac{3}{4} .
    \end{equation*}
Therefore, by \eqref{E:|u-v|/|t u - v|.qst}, 
    \begin{equation*}
     X \leq \frac{|u-v|}{|t u - v|} \leq \frac{1}{1 - 3/4} = 4 
    \end{equation*}
and a Lipschitz inequality is satisfied in this case.

Finally, assume $tq \leq s \leq tq + p/3$. Then, by \eqref{E:pqr.ineqs.Case3}, 
$2r \leq s \leq tq + p/3 < 4p/3$. Hence, $r < 2/3p$ and $p - r \geq p/3$. 
Hence, $p/3 \leq K \xi(x,y)$. Therefore, by \eqref{E:pqr.ineqs.Case3} again and \eqref{E:t:=.threeess/R}, 
    \begin{equation*}
      X = \frac{|u-v|}{\xi(x,y) + |tu - v|} 
        \leq \frac{s+q}{\xi(x,y)} \leq \frac{4p/3 + q}{\xi(x,y)} 
          \leq K \,  \frac{4p/3 + p/t}{p/3} = 3 K(4/3 + t^{-1}) \leq 5 K ,
    \end{equation*}
a Lipschitz inequality. 

Here are the cases we have considered under Case 3, $F^{-1}$:
    \begin{itemize}
    \item $q = 0$ or $s = 0$.
    \item $q, s > 0$:
        \begin{enumerate}
        \item $\gamma \leq 0$,
        \item $\gamma \in [0,1]$:
            \begin{enumerate}
            \item $\sqrt{t} \, q \geq s$,
            \item $s \geq \sqrt{t} \, q$:
                \begin{enumerate}
                \item $\sqrt{t} \, q \leq s \leq tq$,
                \item $s \geq tq + p/3$,
                \item $tq \leq s \leq tq + p/3$.
                \end{enumerate}
            \end{enumerate}
        \end{enumerate}
    \end{itemize}
Taking into account \eqref{E:pqr.ineqs.Case3}, this list is exhaustive. Therefore, we have proved that $F^{-1}$ is Lipschitz in Case 3. 

In this proof we have found multiple, but finitely many, Lipschitz constants. Take the largest one for $F$ and the largest for $F^{-1}$. \emph{This concludes the proof of \eqref{E:F.and.F.invrs.Lip}.}

Now we turn to proving, with the help of \eqref{E:F.and.F.invrs.Lip}, that $f_{dilate, R}$ and its inverse are Lipschitz, the former with Lipschitz constant inversely proportional to $R$. We break up the problem into several cases. 

Recall \eqref{E:Euc.ball.defn}. The set 
$\overline{\clU} \times \overline{B_{1}^{k}(0)}$ is compact. The map 
$(y,v) \mapsto \bigl( y, 2 \rho(y) v \bigr) \in \RR^{2k}$ is continuous on 
$\overline{\clU} \times \overline{B_{1}^{k}(0)}$. Hence, 
$\overline{C_{2} \restriction_{\overline{\clU}}} := \bigl\{ (y,v) : 
y \in \overline{\clU}, |v| \leq 2 \rho(y) \bigr\}$ is compact. Therefore, 
    \begin{equation*}
      \mcl{K} := Exp \bigl( C_{2} \restriction_{\overline{\clU}} \bigr) 
        \subset \mcl{C}  
    \end{equation*}
is compact. By \eqref{E:when.f.dilate.is.identity} and \eqref{E:over.Pf-U.fdil.is.identity}, 
    \begin{equation}  \label{E:fdil.is.ident.on.C-K}
      \text{If } x \in \mcl{D} \setminus \mcl{K} \text{ then } f_{dilate, R}(x) = x . 
    \end{equation}
(Note that $\mcl{D} \setminus \mcl{C} \subset \mcl{D} \setminus \mcl{K}$.)   

By \eqref{E:Pf.loc.compct.strat.space}, $\Pf$ is locally compact and, by \eqref{E:overline.U.defn} and \eqref{E:W.U.T.subset}, $\overline{\clU}$ is compact. Therefore, there exists an open set $\clU^{+} \subset \Pf$ with compact closure s.t.\ 
$\overline{\clU} \subset \clU^{+}$. It follows from definition \ref{D:fibering.by.cones}, part \ref{I:Exp.alpha.homeom} that $C[\Pf]$ is open. Recall the definition of $\pi_{C}$ from definition \ref{D:fibering.by.cones}. We have that 
$C[\clU^{+}] = \pi_{C}^{-1}(\clU^{+})$ 
is open in $T \D \restriction_{\Pf}$. But, by definition \ref{D:fibering.by.cones}, 
part \ref{I:Exp.alpha.homeom} again, $Exp \restriction_{C[\Pf]}$ is a homeomorphism. 
Therefore, 
    \begin{equation}  \label{E:mcl.V.defn}
      \mcl{V} := Exp \bigl( C[\clU^{+}] \bigr) \subset \mcl{C} \subset \D
    \end{equation} 
is open. Now, $C_{2} \restriction_{\overline{\clU}} \subset C[\clU^{+}$. Therefore, 
    \begin{equation*}
      \mcl{K} \subset \mcl{V} \text{ so }
        \mcl{C} \setminus \mcl{V} \subset \mcl{C} \setminus \mcl{K} .
    \end{equation*} 

Let $x, x' \in \D$. We wish to bound 
$\xi \bigl( f_{dilate, R}(x), f_{dilate, R}(x') \bigr)$ by something proportional 
to $\xi(x,x')/R$ and do something similar for $f_{dilate, R}^{-1}$ (but without the 
``$/R$'' part). We first focus on $f_{dilate, R}$ and do the inverse later. 

We can write  
$\D = \mcl{K} \cup (\mcl{C} \setminus \mcl{K}) \cup (\D \setminus \mcl{C})$. The following set of cases is complete. 

  \begin{enumerate}
\item $x \in \mcl{K}$, $x' \in \mcl{K}$.  \label{I:K,K} 
\item $x \in \mcl{K}$, $x' \in \mcl{C} \setminus \mcl{K}$.  \label{I:K,C-K}  
\item $x \in \mcl{K}$, $x' \in \mcl{D} \setminus \mcl{C}$.  \label{I:K,D-C}  
\item $x \in \mcl{C} \setminus \mcl{K}$, $x' \in \mcl{C} \setminus \mcl{K}$.    
\item $x \in \mcl{C} \setminus \mcl{K}$, $x' \in \mcl{D} \setminus \mcl{C}$.     
\item $x \in \mcl{D} \setminus \mcl{C}$, $x' \in \mcl{D} \setminus \mcl{C}$.  
  \end{enumerate}  
  
We have $\mcl{K} \subset \mcl{V} \subset \mcl{C}$. Therefore, 
$\mcl{C} \setminus \mcl{V} \subset \mcl{C} \setminus \mcl{K}$. Let ``${}^{c}$'' denote complementation relative to $\D$. Then 
    \begin{multline*}
      (\mcl{C} \setminus \mcl{K}) \setminus (\mcl{C} \setminus \mcl{V})
        = (\mcl{C} \cap \mcl{K}^{c}) \cap (\mcl{C} \cap \mcl{V}^{c})^{c}
          = (\mcl{C} \cap \mcl{K}^{c}) \cap (\mcl{C}^{c} \cup \mcl{V}) \\
            = (\mcl{C} \cap \mcl{K}^{c}) \cap \mcl{V} 
              =  \mcl{C} \cap (\mcl{V} \setminus \mcl{K}) 
                = \mcl{V} \setminus \mcl{K} \subset \mcl{V} .
    \end{multline*} 
Therefore, if $x \in \mcl{K}$ and 
$x' \in (\mcl{C} \setminus \mcl{K}) \setminus (\mcl{C} \setminus \mcl{V})$ then 
$x \in \mcl{V}$ and $x' \in \mcl{V}$. Hence, cases \ref{I:K,K} and \ref{I:K,C-K} are subsumed under  
    \begin{align*}
      x \in \mcl{V}, & \; x' \in \mcl{V}.  \\ 
      x \in \mcl{K}, & \; x' \in \mcl{C} \setminus \mcl{V}.
    \end{align*}

Now consider case \ref{I:K,D-C} in conjunction with the second case in the preceding. Since $\mcl{V} \subset \mcl{C}$, we have 
    \begin{equation}  \label{E:(C-V)+(D-C)=(D-V)}
      (\mcl{C} \setminus \mcl{V}) \cup (\mcl{D} \setminus \mcl{C}) 
        = \mcl{V}^{c} .
    \end{equation} 

So we can replace our list of six cases by the following five:
  \begin{enumerate}
\item $x \in \mcl{V}$, $x' \in \mcl{V}$.  \label{I:V,V} 
\item $x \in \mcl{K}$, $x' \in \mcl{V}^{c}$.  \label{I:V,D-V}  
\item $x \in \mcl{C} \setminus \mcl{K}$, $x' \in \mcl{C} \setminus \mcl{K}$.  
  \label{I:C-K,C-K}  
\item $x \in \mcl{C} \setminus \mcl{K}$, $x' \in \mcl{D} \setminus \mcl{C}$.  
  \label{I:C-K,D-C}   
\item $x \in \mcl{D} \setminus \mcl{C}$, $x' \in \mcl{D} \setminus \mcl{C}$.  
  \label{I:D-C,D-C}
  \end{enumerate}

\emph{Case \ref{I:V,V}:} Since $\clU^{+}$ is relatively compact, by definition \ref{D:fibering.by.cones}, part \ref{I:Exp.alpha.homeom} yet again, $\alpha$ and its inverse, which is just 
$Exp \restriction_{C[\Pf]}$, are Lipschitz on $\mcl{V}$ and $C[\clU^{+}]$, resp. 

But, by \eqref{E:F.on.parts}, $F \bigl( C[\clU^{+}] \bigr) = C[\clU^{+}]$ and, by \eqref{E:f.in.terms.of.F}, 
$f_{dilate, R} = \alpha^{-1} \circ F \circ \alpha$ on $\mcl{V}$. Thus, by \eqref{E:F.and.F.invrs.Lip}, the restriction $f_{dilate, R} \restriction_{\mcl{V}}$ and its inverse are Lipschitz. In fact, if $K_{\alpha} < \infty$ is a Lipschitz constant for 
$\alpha$ and $\alpha^{-1}$ and $K_{F} < \infty$ is s.t.\ $K_{F}/R$ is a Lipschitz constant for $F_{R}$ ($0 < R < \threeess/3$), then by \eqref{E:comp.of.Lips.is.Lip}, 
the Lipschitz constant for $f_{dilate, R} \restriction_{\mcl{V}}$ is 
    \begin{multline}  \label{E:a.Lip.sandwich}
      (\text{Lipschitz constant for } \alpha) (\text{Lipschitz constant for } F) 
        (\text{Lipschitz constant for } \alpha^{-1}) \\
          = K_{\alpha}^{2} K_{F}/R .
    \end{multline}
 I.e., the Lipschitz constant for $f_{dilate, R} \restriction_{\mcl{V}}$ is inversely proportional to $R$, as asserted by lemma \ref{L:f.dil.Lipschitz}. That takes care of case \ref{I:V,V} in the revised list of cases. 
 
 \emph{Case \ref{I:V,D-V}:} Let $\epsilon := dist(\mcl{V}^{c}, \mcl{K})/2$. 
Since $\mcl{V}^{c}$ is closed and $\mcl{K}$ compact, we have $\epsilon > 0$.\footnote{Suppose $\epsilon = 0$, i.e.\  
$dist(\mcl{V}^{c}, \mcl{K}) = 0$. Then, by \eqref{E:set.distances}, for every $m = 1, 2, \ldots$, there exist $x_{m} \in \mcl{K}$ and 
$z_{m} \in \mcl{V}^{c}$ s.t.\ $\xi(z_{m} , x_{m} ) < 1/m$. Since $\mcl{K}$ is compact, there exists $x_{\infty} \in \mcl{K}$ and a subsequence $\{m_{j}, j = 1, 2, \ldots\}$ s.t.\ $x_{m_{j}} \to x_{\infty}$. Given 
$k = 1, 2, \ldots$, define $\ell_{k}$ to be the smallest $m_{j} > 2k$ s.t.\ 
$\xi(x_{m_{j}}, x_{\infty}) \leq 1/(2k)$. Then $\xi(z_{\ell_{k}}, x_{\infty}) < 1/k$. Thus, $x_{\infty}$ is a limit point of the closed set $\mcl{V}^{c}$ and, hence, $x_{\infty} \in (\mcl{V}^{c}) \cap \mcl{K}$. But $\mcl{V}$ is an open set containing $\mcl{K}$. That contradiction proves $\epsilon > 0$.}
Let $x \in \mcl{K}$ and let $x' \in \mcl{V}^{c}$. 

Suppose $\xi(x,x') < \epsilon$. Then, by definition 
of $\epsilon$, we have $x' \notin \mcl{K}$. Therefore, by \eqref{E:fdil.is.ident.on.C-K}, $f_{dilate, R}(x) = x$. By \eqref{E:(C-V)+(D-C)=(D-V)}, either $x' \in \mcl{C} \setminus \mcl{V}$ or $x' \in \mcl{D} \setminus \mcl{C}$. Suppose the former. Then, by \eqref{E:mcl.V.defn}, 
$\pi_{C} \circ \alpha(x) \in \Pf \setminus \clU^{+} \subset \Pf \setminus \clU$. Hence, by \eqref{E:over.Pf-U.fdil.is.identity}, we have $f_{dilate, R}(x') = x'$. 
Suppose $x' \in \mcl{D} \setminus \mcl{C}$. Then by \eqref{E:when.f.dilate.is.identity}, again $f_{dilate, R}(x') = x'$. Thus, if $\xi(x,x') < \epsilon$, we get a Lipschitz inequality with Lipschitz constant $1< 1/R$, by \eqref{E:R<threeess}. 

Now consider case \ref{I:V,D-V} with $\xi(x,x') \geq \epsilon$. By \eqref{E:t:=.threeess/R}, 
    \begin{multline}  \label{E:long.dist.f.dil.Lip}
      \xi \bigl[ f_{dilate, R}(x), f_{dilate, R}(x') \bigr] 
        < \frac{diam(\D)}{\epsilon} \xi(x,x') 
          < \frac{\threeess \, diam(\D)}{3 R \, \epsilon} \; \xi(x,x') . \\
    \end{multline}
Thus, if $\xi(x,x') \geq \epsilon$, we have a Lipschitz inequality with Lipschitz constant 
$1< 1/R$.

 \emph{Case \ref{I:C-K,C-K}:} By \eqref{E:over.Pf-U.fdil.is.identity}, in this case 
 $f_{dilate, R}(x) = x$ and $f_{dilate, R}(x') = x'$. Thus, we again get a Lipschitz inequality with Lipschitz constant $1< 1/R$.
 
\emph{Case \ref{I:C-K,D-C}:} By \eqref{E:when.f.dilate.is.identity} and \eqref{E:when.f.dilate.is.identity}, $f_{dilate, R}(x) = x$ and $f_{dilate, R}(x') = x'$.

\emph{Case \ref{I:D-C,D-C}:} By \eqref{E:when.f.dilate.is.identity}, 
$f_{dilate, R}(x) = x$ and $f_{dilate, R}(x') = x'$.
  
We have shown that $f_{dilate, R}$ is Lipschitz on $\D$ with Lipschitz constant inversely proportional to $R$. The proof for $f_{dilate, R}^{-1}$, using the fact that $F^{-1}$ is Lipschitz, is almost identical, except for the $1/R$ part (which for 
$f_{dilate, R}$ only is only essential in case \ref{I:V,V}).
 \end{proof}

\begin{proof}[Proof of lemma \ref{L:rect.set.Jake.bounds}]
By Hardt and Simon \cite[Definition 2.1', p.\ 20]{rHlS86.GMT} there is a set 
$N_{0} \subset M$ and an at most countable collection of $\Xdim$-dimensional imbedded $C^{1}$ manifolds $N_{1}, N_{2}, \ldots \subset M$ s.t.\ 
$\Hm^{\Xdim}(N_{0}) = 0$ and $X = \bigcup_{i \geq 0} N_{i}$. Let $x \in N_{i}$ for some 
$i > 0$. By Hardt and Simon \cite[2.6 p.\ 23]{rHlS86.GMT}, the ``approximate tangent space'', $T_{x} X$, of $X$ at $x$ exists and is just the ordinary tangent space 
$T_{x} N_{i}$. We may think of  
	\begin{equation*}
		W := T_{x} X
	\end{equation*}
as a linear subspace of $\RR^{n_{1}}$ of dimension $\Xdim$. A crucial point is that 
$W \subset T_{x} M$. 

Let 
    \begin{equation}  \label{E:Dg(y):Ty.M}
      D g(y) : T_{y} M \to \RR^{n_{2}}
    \end{equation} 
be the derivative or differential of $g$ at points $y$ in the open neighborhood of $M$ in which $g$ is defined. It is a linear map. Regard $T_{x} M$ as an $m$-dimensional linear subspace of $\RR^{n_{1}}$. 
By Hardt and Simon \cite[1.5, p.\ 12 and 1.6, p. 13]{rHlS86.GMT}, we have 
    \begin{equation}  \label{E:dXg=Dg}
      d^{X}g_{x} = Dg(x) \restriction_{T_{x} X} .
    \end{equation}

Recall that if $f : V \to V'$ is a linear map between finite dimensional normed vector spaces then its ``operator norm'' is defined by $\|f\| := \max \bigl\{ |f(v)| : v \in V, |v| = 1 \bigr\}$, where $| \cdot |$ denotes the norm appropriate to the context, $V'$ or $V$.  
From Federer \cite[Section 3.2.1, p.\ 241]{hF69}, 
we see that
	\begin{equation}  \label{E:JXg.in.terms.of.Wedge.Dg}
	      J^{X}g(x) = \Bigl\| \WdgeXdim \bigl[ Dg(x) \restriction_{W} \bigr] \Bigr\|.
	\end{equation}
$\WdgeXdim \bigl[ Dg(x) \restriction_{W} \bigr]$ is obtained from $Dg(x) \restriction_{W}$ as described in Lang \cite[p.\ 426]{sL65.Algebra}. 
By assumption the Riemannian metric on $M$ is induced by that 
on $\RR^{n_{1}}$. Hence, the norm $\| \cdot \|$ is defined based on the inner products 
on $\RR^{n_{1}}$ and $\RR^{n_{2}}$ as described in Federer \cite[Sections 1.7.5, 1.7.6, pp. 31--34]{hF69}.  

We will bound 
$\bigl\| \WdgeXdim \bigl[ Dg(x) \restriction_{W} \bigr] \bigr\|$ above and below. Let 
	\begin{equation}  \label{E:V.and.V'}
		V := T_{x} M \text{ and } V' := T_{g(x)} \RR^{n_{2}} \isomto \RR^{n_{2}} ,
	\end{equation}
so $W \subset V$. Let
	\begin{equation*}
		j : W \to V \text{ be inclusion.}
	\end{equation*}
Write
	\begin{equation} \label{E:k.=.Dg.on.W}
		k = Dg(x) \restriction_{W} = Dg(x) \circ j : W \to V'. 
	\end{equation}
Thus, by \eqref{E:JXg.in.terms.of.Wedge.Dg},  
    \begin{equation}  \label{E:JX.g=|Wedge.k|}
      J^{X}g(x) = \| \WdgeXdim k \| .
    \end{equation}

Use ``${}^{\ast}$'' instead of ``${}^{adj}$'' to indicate adjoint of linear operators (Federer \cite[1.7.4, p.\ 30]{hF69}). $\WdgeXdim k$ is just a linear map from $W_{\Xdim} := \WdgeXdim W$ to $V'_{\Xdim} := \WdgeXdim V'$. \emph{Claim:} 
   \begin{equation}  \label{E:squar.of.Wedge.norm.=.norm.of.Wedge.square}
      \| \WdgeXdim k \|^{2} = \bigl\| \WdgeXdim (k^{\ast} \circ k) \bigr\|.
   \end{equation}
To see this, first note that by Lang \cite[p.\ 426]{sL65.Algebra} again and Federer 
\cite[1.7.6, p.\ 33]{hF69}, 
	\begin{equation}  \label{E:wedge.of.adjoint.compo.=.adjoint.wedge.compo}
		      \WdgeXdim (k^{\ast} \circ k) = ( \WdgeXdim k^{\ast} ) 
		        \circ ( \WdgeXdim k ) 
		           = ( \WdgeXdim k )^{\ast} \circ ( \WdgeXdim k ).
	\end{equation} 
Note that $\WdgeXdim (k^{\ast} \circ k)$, etc.,\ in the preceding map 
$W_{\Xdim}$ into itself. 

Let $h : W_{\Xdim} \to V'_{\Xdim}$ be any linear map.  
By \eqref{E:wedge.of.adjoint.compo.=.adjoint.wedge.compo}, to prove \eqref{E:squar.of.Wedge.norm.=.norm.of.Wedge.square} it suffices to show $\| h^{\ast} \circ h \| = \| h \|^{2}$. Since $W_{\Xdim}$ is a finite dimensional vector space (Federer \cite[p.\ 15]{hF69}), there exists $\mbf{v} \in W_{\Xdim}$ s.t.\ $|\mbf{v}| = 1$ and $\| h \| = | h(\mbf{v}) |$. Use $\langle \cdot , \cdot \rangle$ generically to denote inner products. Then, by the (Cauchy-)Schwarz inequality (Stoll and Wong \cite[Theorem 3.1, p.\ 79] {rrSetW68.LinearAlgebra}), 
   \begin{equation}  \label{E:squar.of.norm.leq.norm.of.square}
      \bigl\| h \bigr\|^{2} = \bigl\langle h(\mbf{v}), h(\mbf{v}) \bigr\rangle 
      = \bigl\langle (h^{\ast} \circ h)(\mbf{v}), \mbf{v} \bigr\rangle  
      \leq \bigl| (h^{\ast} \circ h)(\mbf{v}) \bigr| | \mbf{v} |
          \leq \| h^{\ast} \circ h \|. 
   \end{equation}
Thus, $\| h \|^{2} \leq \| h^{\ast} \circ h \|$. 

Now let $\mbf{v} \in W_{\Xdim}$, $|\mbf{v}| = 1$ be arbitrary. Then, by Federer \cite[1.7.6, p.\ 33]{hF69} again, $\| h^{\ast} \| = \| h \|$ and 
   \begin{equation}   \label{E:norm.of.square.leq.square.of.norm}
      \bigl| (h^{\ast} \circ h)(\mbf{v}) \bigr| 
        \leq \| h^{\ast} \| \bigl| h(\mbf{v}) \bigr| \leq \| h^{\ast} \| \| h \| |\mbf{v}|
          = \| h \|^{2} |\mbf{v}| = \| h \|^{2} .
   \end{equation}
Therefore, $\| h^{\ast} \circ h \| \leq \| h \|^{2}$. With $h = \WdgeXdim k$, the claim \eqref{E:squar.of.Wedge.norm.=.norm.of.Wedge.square} follows. Let 
	\begin{equation}  \label{E:K.=.k.star.k}
		K := (k^{\ast} \circ k) : W \to W. 
	\end{equation}
Then, by \eqref{E:JX.g=|Wedge.k|}, 
    \begin{equation}  \label{E:JX.g.sqrd=|Wedge.K|}
      J^{X}g(x)^{2} = \| \WdgeXdim K \|
    \end{equation}
Now, $\WdgeXdim K : \WdgeXdim W \to \WdgeXdim W$ and $\dim W = \Xdim$. By Federer \cite[1.4.3, p.\ 19]{hF69}, $W_{\Xdim} := \WdgeXdim W$ is a one dimensional space. 

Recall \eqref{E:Dg(y):Ty.M} and \eqref{E:V.and.V'}. Let 
    \begin{equation}  \label{E:G=Dg.adj.circ.Dg}
      G := (Dg(x))^{\ast} \circ (Dg(x)) : V \to V .
    \end{equation} 
($G$ and $K$ are similar but have different domains and codomains.) Then, by \eqref{E:K.=.k.star.k} and \eqref{E:k.=.Dg.on.W}, 
    \begin{equation}  \label{E:K.G.dMg}
      K = j^{\ast} \circ G \circ j 
        = (d^{X}g_{x})^{\ast} \circ (d^{X}g_{x}) .
    \end{equation} 
Moreover, by Lang \cite[p.\ 426]{sL65.Algebra},
$\WdgeXdim K = (\WdgeXdim j^{\ast}) \circ (\WdgeXdim G) \circ (\WdgeXdim j)$. 
Let $\mbf{w} \in W_{\Xdim}$ have norm 1. Then, by Federer \cite[p.\ 32]{hF69}, 
$\mbf{w} = w_{1} \wedge \cdots \wedge w_{\Xdim}$, where $w_{1}, \ldots, w_{\Xdim}$ is any orthonormal basis of $W$. Then $\mbf{w}$ spans $W_{\Xdim}$ and 
$\| \WdgeXdim K \| = \bigl| \WdgeXdim K (\mbf{w}) \bigr|$. Then,
	\begin{equation}  \label{E:norm.wedge.K.=.|wedge.G(x)|}
		\| \WdgeXdim K \| = \bigl| (\WdgeXdim K)(\mbf{w}) \bigr| 
		    = \bigl| (\WdgeXdim j^{\ast}) 
		       \circ (\WdgeXdim G) \circ (\WdgeXdim j)(\mbf{w}) \bigr|.
	\end{equation}

Now, $\dim V = m$. By \eqref{E:G=Dg.adj.circ.Dg}, $G$ is self-adjoint. 
Let 
    \begin{equation*}
      v_{1}, \ldots, v_{m} \in V \text{ be orthonormal eigenvectors of } 
        G = Dg(x)^{\ast} \circ Dg(x)^{\ast} 
    \end{equation*} 
with corresponding eigenvalues 
    \begin{equation}  \label{E:eigvals.of.G}
      \nu_{1}^{2}(x) \geq \cdots \geq \nu_{m}^{2}(x) \geq 0 .
    \end{equation} 
So $v_{1}, \ldots, v_{m}$ is a basis of $V$. Let $\Lambda(m,\Xdim) := \bigl\{ \text{all increasing maps of } \{ 1, \ldots, \Xdim \}$ 
into $\{ 1, \ldots, m \} \bigr\}$. 
By Federer \cite[pp.\ 14 -- 15 and p.\ 32]{hF69} 
    \begin{equation*}
      \mbf{v}_{\mbf{i}} := v_{\mbf{i}(1)} \wedge \cdots \wedge v_{\mbf{i}(\Xdim)}
        \quad (\mbf{i} \in \Lambda(m,\Xdim)
    \end{equation*}
is an orthonormal basis of $V_{\Xdim}$. If $\mbf{i} \in \Lambda(m,\Xdim)$, we have, by Lang \cite[p.\ 426]{sL65.Algebra},
   \begin{align}    \label{E:eigen.wedge.G}
       (\WdgeXdim G)(\mbf{v}_{\mbf{i}})   
        &=  G(\mbf{v}_{\mbf{i}(1)}) \wedge \cdots \wedge G(\mbf{v}_{\mbf{i}(\Xdim)})  
           =  (\nu_{\mbf{i}(1)}^{2} \mbf{v}_{\mbf{i}(1)}) \wedge 
              \cdots \wedge (\nu_{\mbf{i}(\Xdim)}^{2} \mbf{v}_{\mbf{i}(\Xdim)})  \\
        &= \left( \prod_{t=1}^{\Xdim} \nu_{\mbf{i}(t)}^{2} \right) \mbf{v}_{\mbf{i}}. \notag
   \end{align}
Write
	\begin{equation}  \label{E:nu.bold.i.defn}
		 \nu_{\mbf{i}}^{2} := \prod_{t=1}^{\Xdim} \nu_{\mbf{i}(t)}^{2} .
	\end{equation}
Thus, 
    \begin{multline} \label{E:eigs.of.Wedge.G}
      \mbf{v}_{\mbf{i}} \; (\mbf{i} \in \Lambda(m,\Xdim)) 
        \text{ are orthonormal eigenvectors of } 
          \WdgeXdim G. \\
           \text{ The eigenvalue corresponding to $\mbf{v}_{\mbf{i}}$ is }
            \nu_{\mbf{i}}^{2} .
    \end{multline} 
 Write $\mbf{w} = w_{1} \wedge \cdots \wedge w_{\Xdim}$ as above ($\mbf{w}$ spans $W_{\Xdim}$), then $(\WdgeXdim j^{\ast})(\mbf{w}) = j^{\ast}(w_{1}) \wedge 
 \cdots \wedge j^{\ast}(w_{\Xdim}) = w_{1} \wedge \cdots \wedge w_{\Xdim} 
 = \mbf{w}$. 
 Thus, 
     \begin{equation}  \label{E:Wedge.j.is.inclusion}
      \WdgeXdim j : V_{\Xdim} \to W_{\Xdim} \text{ is just inclusion.}
    \end{equation}

By Stoll and Wong \cite[Theorem 6.3(a), p.\ 131]{rrSetW68.LinearAlgebra}, 
$j^{\ast} : V \to W$ is just orthogonal projection onto $W$. 
\emph{Claim:} $\WdgeXdim j^{\ast} :V_{\Xdim} \to W_{\Xdim}$ is orthogonal projection onto $W_{\Xdim}$. We have  
    \begin{equation*}
      (\WdgeXdim j^{\ast}) (\mbf{w})
        = j^{\ast}(w_{1}) \wedge \cdots \wedge j^{\ast}(w_{\Xdim}) 
          = w_{1} \wedge \cdots \wedge w_{\Xdim} = \mbf{w} .
    \end{equation*} 
Suppose $\xi \in W_{\Xdim}$
with $\xi \perp W_{\Xdim}$. Now, by Federer \cite[1.7.6, p.\ 33]{hF69} again, 
$\WdgeXdim j^{\ast} = (\WdgeXdim j)^{\ast}$. Therefore, 
	\begin{equation*}
		0 = \langle \mbf{w}, \xi \rangle 
		  = \bigl\langle (\WdgeXdim j)(\mbf{w}), \xi \bigr\rangle 
		    = \bigl\langle \mbf{w}, (\WdgeXdim j)^{\ast}(\xi) \bigr\rangle
		      = \bigl\langle \mbf{w}, (\WdgeXdim j^{\ast})(\xi) \bigr\rangle.
	\end{equation*} 
Since $W_{\Xdim}$ is spanned by $\mbf{w}$, this proves the claim.

Thus, since $\WdgeXdim j^{\ast}$ is orthogonal projection onto $W_{\Xdim}$ and $\mbf{w}$ spans $W_{\Xdim}$, for any $\xi \in V_{\Xdim}$, there exists $c \in \RR$ s.t.\ 
$(\WdgeXdim j^{\ast})(\xi) = c \mbf{w}$. But $|\mbf{w}| = 1$. Therefore
$c = \bigl\langle \mbf{w}, \; (\WdgeXdim j^{\ast})(\xi) \bigr\rangle$. Hence, by \eqref{E:Wedge.j.is.inclusion}, 
    \begin{equation*}
        (\WdgeXdim j^{\ast})(\xi) = \Bigl\langle \mbf{w}, \; 
          (\WdgeXdim j^{\ast})(\xi) \Bigr\rangle \mbf{w}
          = \Bigl\langle (\WdgeXdim j)(\mbf{w}), \; \xi \Bigr\rangle \mbf{w} 
            = \langle \mbf{w}, \; \xi \rangle \, \mbf{w} .
    \end{equation*}
Consequently,
	\begin{equation} \label{E:norm.from.innr.prod}
		\bigl| (\WdgeXdim j^{\ast})(\xi) \bigr| 
		  = \bigl| \langle \mbf{w}, \; \xi \rangle \bigr| . 
	\end{equation}
	
Since $\mbf{w} \in W_{\Xdim} \subset V_{\Xdim}$, there exists 
$\mbf{a} = \bigl( a_{\mbf{i}}, \mbf{i} \in \Lambda(m,\Xdim) \bigr) 
\in \RR^{\binom{m}{\Xdim}}$ s.t.\ 
    \begin{equation}  \label{E:|a|=1,sum.w}
      | \mbf{a} | = 1 \text{ and } \mbf{w} 
        = \sum_{\mbf{i} \in \Lambda(m,\Xdim)} a_{\mbf{i}} \mbf{v}_{\mbf{i}} .
    \end{equation}
Hence, by \eqref{E:Wedge.j.is.inclusion}, \eqref{E:norm.from.innr.prod}, \eqref{E:|a|=1,sum.w}, and \eqref{E:eigs.of.Wedge.G}, we have
	\begin{align}  \label{E:pythag.for.wedgie.w}
		\bigl| (\WdgeXdim j^{\ast}) \circ (\WdgeXdim G) 
		  \circ (\WdgeXdim j)(\mbf{w}) \bigr| 
		  &= \Bigl| (\WdgeXdim j^{\ast}) 
		    \bigl[ (\WdgeXdim G) \circ (\WdgeXdim j)(\mbf{w}) \bigr] \Bigr| \notag \\ 
		  &= \Bigl| (\WdgeXdim j^{\ast}) 
		    \bigl[ (\WdgeXdim G) (\mbf{w}) \bigr] \Bigr| \notag \\
		  &= \Bigl| \bigl\langle \mbf{w} , 
		     (\WdgeXdim G) (\mbf{w}) \bigr\rangle \Bigr| \\
		 &= \left| \left\langle \sum_{\mbf{i} 
		   \in \Lambda(m,\Xdim)} a_{\mbf{i}} \, \mbf{v}_{\mbf{i}},
		        \sum_{\mbf{i} \in \Lambda(m,\Xdim)} a_{\mbf{i}} \, \nu_{\mbf{i}}^{2} \, 
		            \mbf{v}_{\mbf{i}} \right\rangle \right| \notag \\
		 &=  \sum_{\mbf{i} \in \Lambda(m,\Xdim)} a_{\mbf{i}}^{2} \, 
		         \nu_{\mbf{i}}^{2}. \notag
	\end{align}

Now, $\mbf{w}$ is a unit vector spanning the one-dimensional space $W_{\Xdim}$. Therefore, by \eqref{E:JX.g.sqrd=|Wedge.K|} and \eqref{E:norm.wedge.K.=.|wedge.G(x)|},  
    \begin{equation*}
        \bigl( J^{X}g(x) \bigr)^{2} = \Bigl\| (\WdgeXdim j^{\ast}) 
          \circ (\WdgeXdim G) \circ (\WdgeXdim j) \Bigr\| .
    \end{equation*}
Furthermore, $\WdgeXdim G) \circ (\WdgeXdim j)$ is defined on the one-dimensional space 
$W_{\Xdim}$ and $\mbf{w}$ is a unit vector in that space. Therefore, by \eqref{E:pythag.for.wedgie.w}, 
    \begin{equation}  \label{E:Jake.a's.nu's}
      \bigl( J^{X}g(x) \bigr)^{2}  
          = \bigl| (\WdgeXdim j^{\ast}) \circ (\WdgeXdim G) 
            \circ (\WdgeXdim j)(\mbf{w}) \bigr|
             =  \sum_{\mbf{i} \in \Lambda(m,\Xdim)} a_{\mbf{i}}^{2} \, \nu_{\mbf{i}}^{2}.
    \end{equation}

  Let $\mbf{i}_{0}, \mbf{i}_{1} \in \Lambda(m,\Xdim)$ satisfy 
	\[
	\nu_{\mbf{i}_{0}} := 
	   \sqrt{ \min_{\mbf{i} \in \Lambda(m,\Xdim)} \nu_{\mbf{i}}^{2} }
				  \;  \text{ and } \; 
	\nu_{\mbf{i}_{1}} := 
	   \sqrt{ \max_{\mbf{i} \in \Lambda(m,\Xdim)} \nu_{\mbf{i}}^{2} } .
	 \]
Then, by \eqref{E:nu.bold.i.defn} and \eqref{E:eigvals.of.G}, 
we have 
	\begin{equation}  \label{E:biggest.smallest.lambda.prods}
		\nu_{\mbf{i}_{0}}^{2} = \prod_{t=m-\Xdim+1}^{m} \nu_{t}^{2}
			\;  \text{ and } \;
		\nu_{\mbf{i}_{1}}^{2} =  \prod_{t=1}^{\Xdim} \nu_{t}^{2}.
	\end{equation}

Now, by \eqref{E:|a|=1,sum.w}, 
$\sum_{\mbf{i} \in \Lambda(m,\Xdim)} a_{\mbf{i}}^{2} = 1$ so, by \eqref{E:Jake.a's.nu's} and \eqref{E:biggest.smallest.lambda.prods}, 
   \begin{equation*}
      \prod_{t=m-\Xdim+1}^{m} \nu_{\mbf{i}(t)}^{2} 
      = \left( \sum_{\mbf{i} \in \Lambda(m,\Xdim)} a_{\mbf{i}}^{2} \right) 
        \nu_{\mbf{i}_{0}}^{2} 
          \leq J^{X}g(x)^{2}
     \leq \left( \sum_{\mbf{i} \in \Lambda(m,\Xdim)} a_{\mbf{i}}^{2} \right) 
       \nu_{\mbf{i}_{1}}^{2} = \prod_{t=1}^{\Xdim} \nu_{t}^{2}.
   \end{equation*}
This completes the proof of the lemma. 
\end{proof}

  \begin{proof}[Proof of lemma \ref{L:Pf.is.a.manif}]
 Let $\mcl{K}$ be the set of all matrices of dimension $(n-1) \times \nvar$ having rank $k$. Then, by lemma \ref{L:rank.k.mats.form.manif} below, $\mcl{K}$ is an imbedded submanifold of $\RR^{(n-1) \times \nvar}$ of dimension $(n-1)k + k \nvar - k^{2}$. Define a map $F : \mcl{K} \to \Y$ as follows. If $X^{(n-1) \times \nvar} \in \mcl{K}$, let $y_{n}^{1 \times \nvar}$ equal minus the sum of the rows of $X$ and let $F(X)^{n \times \nvar}$ be the matrix 
	$\begin{pmatrix}
		X \\
		y_{n}
	\end{pmatrix}$. 
So the sum of the rows of $F(X)$ is $0^{1 \times \nvar}$. By Boothby \cite[Exercise 2, p.\ 81]{wmB75}, $F(\mcl{K})$ is an imbedded submanifold of $\Y$ of dimension $(n-1)k + k \nvar - k^{2}$. By Boothby \cite[Theorem (1.7), p.\ 57]{wmB75}, $F(\mcl{K}) \times \RR^{\nvar}$ is a smooth manifold of dimension $nk + (k + 1)(\nvar - k)$. Now consider the one-to-one immersion $G : F(\mcl{K}) \times \RR^{\nvar} \to \Y$ defined by 
$G(Y, b) := Y + 1^{n} b$, where $Y \in F(\mcl{K}), b^{1 \times \nvar} \in \RR^{\nvar}$. ($(1^{n})^{n \times 1}$ is the column matrix of 1's, \eqref{E:1n.col.vec.defn}.) 
Then, by Boothby \cite[Exercise 2, p.\ 81]{wmB75} again, $G \bigl[ F(\mcl{K}) \times \RR^{\nvar} \bigr]$ is an imbedded submanifold of $\Y$.
\end{proof}

   \begin{lemma}  \label{L:rank.k.mats.form.manif}
Let $0 < k < \nvar \leq n$ be integers. Let $\mcl{K} \subset \RR^{n \nvar}$ be the set of all 
$n \times \nvar$ matrices of rank exactly $k$. Then $\mcl{K}$ is an imbedded smooth submanifold of $\RR^{n \nvar}$ of dimension $nk + k \nvar - k^{2} < n \nvar = $ dimension of space of all $n \times \nvar$ matrices.
   \end{lemma}
  \begin{proof}
 If $0 < j_{1} < \ldots < j_{k} \leq \nvar$ are integers, write $\mbf{j} := \{  j_{1}, \ldots, j_{k} \}$. If $N$ is $n \times \nvar$, 
let $N_{\mbf{j}}^{n \times k}$ be the matrix whose $i^{th}$ column is column $j_{i}$ of $N$ ($i = 1, \ldots, k$)

Let $M \in \mcl{K}$. There exists $\mbf{j}$ s.t.\ $M_{\mbf{j}}$ has rank $k$. 
For $i \in \NN_{n}$ let $e_{i}^{n \times 1}$ be the column vector whose $i^{th}$ coordinate is $1$ and whose other coordinates are 0. 

\emph{Claim:} There are integers $0 < i_{k+1} < \cdots < i_{n} \leq n$ s.t.\ the matrix 
$N := (M_{\mbf{j}}, E)^{n \times (n - k)}$ has rank $n$, where $E^{n \times (n-k)}$ is the matrix $(e_{i_{k+1}}, \ldots, e_{i_{n}})$. To prove this, put $M_{\mbf{j}}^{T}$ into echelon form (Stoll and Wong \cite[p.\ 46]{rrSetW68.LinearAlgebra}). Call the result $L^{k \times n}$. It has the same row space as $M^{T}$. Then, by Stoll and Wong \cite[Theorem 3.1, p.\ 47]{rrSetW68.LinearAlgebra}, reading left to right, each nonzero row of $L$ begins with 0's followed by a 1, the ``leading 1'' of the row. Moreover, except for the ``1'' all other entries in the column containing a leading 1 are 0. Since $L$ has rank $k$, all the rows of $L$ are nonzero. Let $0 < h_{1} < \cdots < h_{k} \leq n$ be the numbers of the columns of $L$ containing leading 1's. Let $i_{k+1} < \cdots < i_{n}$ be the indices in $1, \ldots, n$ not included in $0 < h_{1} < \cdots < h_{k} \leq n$. Then the columns of the corresponding $E$ are linearly independent but not in the span of the columns of $M$. This proves the claim.

For $i = 1, \ldots, \nvar$, let $f_{i}$ be the column $\nvar$-vector that is all 0, except for a 1 in the $i^{th}$ position. Extend $\mbf{j}$ to be a permutation, $i \mapsto j_{i}$, of $1, \ldots, \nvar$. For definiteness, assume $0 < j_{k+1} < \ldots < j_{\nvar} \leq \nvar$. Let $P_{\mbf{j}}^{\nvar \times \nvar}$ be the matrix whose $j_{i}^{th}$ column is $f_{i}$. Then there exists a unique matrix $C_{1}^{k \times (\nvar - k)}$ s.t.\ 
	\begin{equation*}
		M = M_{\mbf{j}} (I_{k}, C_{1}) P_{\mbf{j}}.
	\end{equation*}
(Thus, column $j_{i}$ of $(I_{k}, C_{1}) P_{\mbf{j}}$ is just $f_{i}$, $i=1, \ldots, k$.)

By lemma \ref{L:rank.lwr.semicont}, $M_{\mbf{j}}$ has a neighborhood $V_{M_{\mbf{j}},k}^{n \times k} \subset \RR^{nk}$ 
s.t.\ $L \in V_{M_{\mbf{j}},k}$ implies that $(L, E)^{n \times n}$ has rank $n$. 
Let $V_{E} := V_{M_{\mbf{j}}, \mbf{j}} := V_{M_{\mbf{j}},k} \times \RR^{k(\nvar-k)} \times \RR^{(n-k)(\nvar-k)}$. Then $V_{E}$ is an open subset 
of $\RR^{n \nvar}$. Define $\psi := \psi_{M_{\mbf{j}}, \mbf{j}} : V_{E} \to \RR^{n \nvar}$ by
	\begin{equation*}
		\psi \bigl( L^{n \times k}, D_{1}^{k \times (\nvar - k)}, 
		  D_{2}^{(n-k) \times (\nvar - k)} \bigr) 
		   := \bigl[ L \, (I_{k}, \, D_{1}) + (0^{n \times k}, \, E D_{2}) \bigr] P_{\mbf{j}}. 
	\end{equation*}
Then $\psi$ is a smooth map of $V$ into the set 
$\mcl{N} := \mcl{N}_{M_{\mbf{j}}, \mbf{j}} := \bigl\{ N^{n \times \nvar} \in \RR^{n \nvar} : N_{\mbf{j}} \in V_{M_{\mbf{j}},k} \bigr\}$, an open subset of $\RR^{nq}$. 

Let $N \in \mcl{N}$ and let $L = N_{\mbf{j}}$. Then for some $D^{n \times (\nvar-k)}$, we have $N = (L, D) P_{\mbf{j}}$. 
Let $N_{2}^{n \times (\nvar-k)}$ be the matrix consisting of the last $\nvar - k$ columns of $N P_{\mbf{j}}^{-1}$. Now, by definition 
of $\mcl{N}$ and $V_{M_{\mbf{j}},k}$, we have that $(L, E)^{n \times n}$ has rank $n$. 
Therefore, $D^{n \times (\nvar-k)} = (L, E)^{-1} N_{2}$ is valid. Write
    $D = \left(
        \begin{smallmatrix}
            D_{1}^{k \times (\nvar - k)} \\
            D_{2}^{(n-k) \times (\nvar - k)}
        \end{smallmatrix}
    \right)$.
Then $N = \psi(L, D_{1}, D_{2})$. This shows that $\psi : V_{E} \to \mcl{N}$ is onto and $\psi$ has an inverse, 
$\varphi := \varphi_{M_{\mbf{j}}, \mbf{j}}$. Both $\psi$ and $\varphi$ are compositions of matrix operations and, hence, are smooth.

Let $M' \in \mcl{K}$. Let $\mbf{j}'$ be a $k$-tuple $0 < j'_{1} < \ldots < j'_{k} \leq \nvar$ of integers s.t.\ 
the matrix $(M_{\mbf{j}'}')^{n \times k}$ has rank $k$. Let $E'$ be the corresponding $n \times (n-k)$ matrix of 0's and 1's. 
Suppose $V_{M_{\mbf{j}}, \mbf{j}} \cap V_{M_{\mbf{j}'}', \mbf{j}'} \neq \varnothing$. Then 
$\varphi_{M_{\mbf{j}'}', \mbf{j}'} \circ \psi_{M_{\mbf{j}}, \mbf{j}}$, on the obvious subset of $\mcl{N}_{M_{\mbf{j}}, \mbf{j}}$, is a composition of matrix operations and is therefore smooth. 
Hence, by Boothby \cite[Theorem (1.3), p.\ 54]{wmB75} the system $\varphi_{M_{\mbf{j}}, \mbf{j}}$ as $\mbf{j}$ and $E$ vary determines a differentiable structure on a neighborhood, call it $\mcl{M}$, of $\mcl{K}$ in $\RR^{nq}$.

But letting $D_{2}$ be identically 0 we see that $\mcl{K}$ has the $(nk + k \nvar - k^{2})$-submanifold property relative to $\mcl{M}$
(Boothby \cite[Definition (5.1), p.\ 75]{wmB75}). Hence, by Boothby \cite[Lemma (5.2), p.\ 76]{wmB75}, $\mcl{K}$ 
is an imbedded $(nk + k \nvar - k^{2})$-submanifold of $\RR^{n \nvar}$.

Finally, $nk + k \nvar - k^{2} = nk + k(\nvar - k) < nk + n(\nvar - k) = n \nvar$.
  \end{proof}
   
  \begin{lemma}  \label{L:rank.lwr.semicont}
Let $m, n = 1, 2, \ldots$ and let $\mcl{M}$ be the space of $m \times n$ matrices. Then the function $rank : \mcl{M} \to \RR$ is lower semicontinuous (Ash \cite[Definition A6.1, p.\ 388]{rbA72}). I.e., for $s \in \RR$, the set $\{ M \in \mcl{M} : rank \, M > s \}$ is open 
in $\mcl{M}$. Moreover, the set of $m \times n$ matrices of full rank $\min(m,n)$ is dense in $\mcl{M}$.
  \end{lemma}
  \begin{proof}
We start with a slavish copy of the argument in Boothby \cite[p.\ 47]{wmB75}. Recall that a square matrix is of full rank if and only if its determinant is nonzero (Stoll and Wong \cite[Theorem 5.6, p.\ 175]{rrSetW68.LinearAlgebra}). Therefore, if $M \in \mcl{M}$ and $k = 1, \ldots, \min \{m,n\}$ then it is easily seen that $rank \, M > k-1$ if and only if some $k \times k$ minor determinant of $M$ is nonzero. Since the vector of all 
$k \times k$ minor determinants of $M$ is continuous in $M$, regarded as a point 
in $\RR^{mn}$

WLOG we may assume $m \leq n$. Let $M \in \mcl{M}$ and suppose $rank \, M < m$.
It follows from the Singular Value Decomposition (Rao \cite[(v), p.\ 42]{crR73.LinStatInf}) that by making arbitrarily small perturbations in $M$, we can get a matrix with full rank.
  \end{proof}

  \begin{proof}[Proof of lemma \ref{L:proj.mat.is.imbedding.of.Grass}]
Some of this proof may be duplicative of material in section \ref{SS:D.T.plane.fit}. Sorry. 

We show that $\Pi$ is a one-to-one immersion (Boothby \cite[Definition (4.3), p.\ 70]{wmB75}). This makes sense since $\dim G(k, \nvar) = k (\nvar-k) < \nvar^{2} = \dim \mcl{M}$. 
Then, since $G(k, \nvar)$ is compact (Milnor and Stasheff \cite[Lemma 5.1, p.\ 57]{jwMjdS74}), 
by Boothby \cite[Theorem (5.7), p.\ 79]{wmB75}, it will follow that $\Pi$ is an imbedding. Since the row space of $\Pi(\xi)$ 
is just $\xi$ ($\xi \in G(k, \nvar)$), $\Pi$ is obviously one-to-one. It remains to show that $\Pi$ is an immersion.

Let $\xi_{0} \in G(k, \nvar)$ and let $U \subset G(k, \nvar)$ be a coordinate neighborhood of $\xi_{0}$ with coordinate map 
$\varphi : U \to \mcl{W}$, the space of matrices $k \times (\nvar-k)$ matrices as described in Boothby \cite[p.\ 64]{wmB75}. 
$\varphi$ depends on a choice of indices $1 \leq j_{1} < \cdots < j_{k} \leq \nvar$. WLOG we take $j_{i} = i$ ($i = 1, \ldots, k$). 
If $W \in \mcl{W}$ let $\mbf{x}^{k \times \nvar} := \mbf{x}(W) := (I_{k}, W)$, where $I_{k}$ is the $k \times k$ identity matrix. Then in the setup described in Boothby \cite[p.\ 64]{wmB75}, if $\xi \in U$ and $W := \varphi(\xi) \in \mcl{W}$ then $\rho \bigl[ \mbf{x}(W) \bigr] = \xi$. 
I.e., $\rho \circ \mbf{x} \circ \varphi$ is the identity map on $U$. For $W \in \mcl{W}$, 
let $\pi(W ) := \mbf{x}(W)^{T} \bigl[ \mbf{x}(W) \mbf{x}(W)^{T} \bigr]^{-1} \mbf{x}(W)$. Thus, by \eqref{E:proj.mat.from.K}, $\pi(W)$ is the matrix of orthogonal projection
onto $\rho \bigl[ \mbf{x}(W) \bigr]$. Since $\pi(W)$ is a rational function in the entries of $W$, we see that $\pi$, 
and hence $\Pi = \pi \circ \varphi$, is smooth. 

We need to show that $\Pi_{\ast}$ is an injection at each point of $G(k, \nvar)$. Notice that $\rho \circ \Pi = id$, where $id$ is the identity on $G(k, \nvar)$. We know that $\rho$ is smooth, by \eqref{E:rho.is.smooth}, 
so we can write $\rho_{\ast} \circ \Pi_{\ast} = id_{\ast}$. Suppose $\Pi_{\ast}$ were not one-to-one. Then there would exist 
a nonzero tangent vector $X \in T_{\xi} \, G(k, \nvar)$ at some $\xi \in G(k,\nvar)$ s.t.\ 
$0 = \rho_{\ast} \circ \Pi_{\ast} (X) = id_{\ast} (X) = X \neq 0$, contradiction. This proves the lemma.
  \end{proof}

  \begin{proof}[Proof of equivalence of two definitions of PC on the plane]
The two definitions are in example \ref{Ex:3.plane.fitters} and section \ref{SSS:LF.plots}.
In section \ref{SSS:LF.plots}, principal components plane-fitting (on the plane) is defined as follows: ``Principal components line fitting (PC) finds the line that minimizes the sum of the squared \emph{perpendicular} distances from points in the cloud to the line.'' The official definition is given in example \ref{Ex:3.plane.fitters}. Here we show, for fitting a line to a point cloud on the plane, the two definitions are essentially equivalent. 

Let $n > 2$ and let $x$ be an $n \times 2$ matrix, each row of which gives the coordinates of a point in a point cloud on the plane. 
Use ``${}^{\centerdot \times \centerdot}$'' notation to indicate matrix dimensions. Thus, e.g., we write $x^{n \times 2}$. $1^{n}$ is defined in \eqref{E:1n.col.vec.defn} to be the column $n$-vector all of whose entries are 1. Then we can write 
    \begin{equation}  \label{E:center.x}
      x = z + 1^{n} \bar{x} ,
    \end{equation} 
where 
$z^{n \times 2}$ satisfies 
    \begin{equation*}
      (1^{n})^{T} z = 0 
    \end{equation*}
and $\bar{x}$ is $1 \times 2$. (The coordinates of $\bar{x}$ are the column means of $x$.) 
Use superscript ``${}^{T}$'' to indicate matrix transposition.  

By the Singular Value Decomposition (Rao \cite[(v), p.\ 42]{crR73.LinStatInf}), we may write 
    \begin{equation}  \label{E:z.L.Lambda.NT}
      z = L^{n \times 2} \Lambda^{2 \times 2} N^{T} ,
    \end{equation} 
where $L$ has orthonormal columns, $\Lambda$ is non-negative diagonal, and 
$N^{2 \times 2}$ is orthogonal. For $j = 1$ or 2 let $u_{j}$ be the $j^{th}$ column of $L$ and $v_{j}$ be the $j^{th}$ column of $N$. 
So $u_{j}$ is an $n \times 1$ column vector and $v_{j}$ is a $2 \times 1$ column vector. The $u_{j}$'s and $v_{j}$'s are unit vectors and $u_{2}^{T} u_{1} = 0$ and 
$v_{2}^{T} v_{1} = 0$. Write 
      \begin{equation*}
        \Lambda =
            \begin{pmatrix}
                  \lambda_{1} & 0 \\
                  0  & \lambda_{2}
            \end{pmatrix} .
      \end{equation*}
\eqref{E:z.L.Lambda.NT} can be rewritten 
    \begin{equation}  \label{E:z.in.trms.of.v1.v2}
      z = \lambda_{1} u_{1} v_{1}^{T} +  \lambda_{2} u_{2} v_{2}^{T} .
    \end{equation} 

Since the columns of $z$ each have mean 0, by \eqref{E:sample.covariance.matrix}, the covariance matrix $C$ of $x$ is 
$(n-1)^{-1} z^{T} z = (n-1)^{-1} N \Lambda L^{T} L \Lambda N^{T} 
= (n-1)^{-1} N \Lambda^{2} N^{T}$. Thus, the eigenvectors of $C$ are $v_{1}$ and $v_{2}$ with eigenvalues $\lambda_{j}^{2}/(n-1)$ $(j=1,2)$. WLOG 
$\lambda_{1} \geq \lambda_{2}$. If the two eigenvalues are equal, then $z$ is a singularity of PC. (See section \ref{SS:PC.plane.fitting}.) So assume the inequality is strict. Thus, according to the example \ref{Ex:3.plane.fitters} definition, the PC line for $x$ is 
$\{ t v_{1}^{T} : t \in \RR \}$. We show that the PC line according to the section \ref{SSS:LF.plots} is parallel to that. In our analysis of plane-fitting we regard parallel lines as the same. So this will establish that the two definitions of PC are the same (at least on the plane).

Let $w^{1 \times 2}$ be a unit row vector and let $r$ be $1 \times 2$. Then 
    \begin{equation*}
      t \mapsto r + t w \text{  parametrizes a generic line, }
    \end{equation*}
call it $\ell(r,w)$, on the plane. We will compute the $w$ and $r$ s.t.\ (such that) the sum of the squared row-wise perpendicular distances from $x$ to that line is minimized. To begin with we will not insist on perpendicular distances, just any distances. Perpendicularity will arise automatically. 

Let $p$ be a column $n$-vector. Then $1^{n} r + p w$ is an $n \times 2$ matrix. Its rows lie on the line $\ell(r,w)$. Write 
    \begin{equation*}
      p = q + s 1^{n} ,
    \end{equation*} 
where $q^{T} 1^{n} = 0$ and $s \in \RR$. Recall $(1^{n})^{T} z = 0$.  We choose $p$, 
$r$, and the unit vector $w$ to minimize the following.  
    \begin{multline*}
      \bigl| x - (1^{n} r + p w) \bigr|^{2} = | z + 1^{n} \bar{x} - p w - 1^{n} r |^{2} \\
       = \bigl| (z - q w ) + 1^{n} ( \bar{x} - s w - r ) \bigr|^{2} 
         = | z - q w |^{2} + n | \bar{x} - s w - r |^{2} .
    \end{multline*}
Setting $r = \bar{x} - s w$ can only decreases this. So we may assume 
$\bigl| x - (1^{n} r + p w) \bigr|^{2} = | z - q w |^{2}$. We will compute $q$. $s \in \RR$ can then be arbitrary. So we take $s = 0$. This makes 
    \begin{equation*}
      p = q \text{ and } r = \bar{x} .
    \end{equation*} 

If $| z - q w |^{2}$ is minimum then, by \eqref{E:z.L.Lambda.NT}, $w$ lies in the row space of $z$, a subspace of the column space of $N$. Thus, 
    \begin{equation*}
      w = c v_{1}^{T} + d v_{2}^{T} , \text{ where } c^{2} + d^{2} = 1 .
    \end{equation*}
By \eqref{E:z.in.trms.of.v1.v2}, we have
    \begin{multline*}
      | z - q w |^{2} = \bigl| (\lambda_{1} u_{1} - cq)v_{1}^{T} 
        + (\lambda_{2} u_{2} - dq)v_{2}^{T} \bigr|^{2}
          = |\lambda_{1} u_{1} - cq|^{2} + |\lambda_{2} u_{2} - d q|^{2} \\
           = \lambda_{1}^{2} +  \lambda_{2}^{2} 
             - 2 (\lambda_{1} c u_{1} + \lambda_{2} d u_{2})^{T} q + (c^{2}+d^{2}) |q|^{2}
               = \lambda_{1}^{2} +  \lambda_{2}^{2} 
                - 2 (\lambda_{1} c u_{1} + \lambda_{2} d u_{2})^{T} q + |q|^{2} .
    \end{multline*}
 
 Let $y :=  \lambda_{1} c u_{1} + \lambda_{2} d u_{2}$. So 
 $| z - q w |^{2} = \lambda_{1}^{2} +  \lambda_{2}^{2} - 2 y^{T} q + |q|^{2}$. Holding $|q|$ fixed, $- 2 (\lambda_{1} c u_{1} + \lambda_{2} d u_{2}) q^{T}$ is made smaller (more negative) by letting $q = e y$ for some $e \geq 0$. Thus, we get
    \begin{multline*}
      | z - q w |^{2} = \lambda_{1}^{2} +  \lambda_{2}^{2} 
        - 2 e |y|^{2}  + e^{2} |y|^{2} \\
          = \lambda_{1}^{2} +  \lambda_{2}^{2} - |y|^{2}
            + \bigl( |y|^{2} - 2 e |y|^{2}  + e^{2} |y|^{2} \bigr)
              = \lambda_{1}^{2} +  \lambda_{2}^{2} - |y|^{2} + |y|^{2} (1 - e)^{2} .
    \end{multline*}
We make this smaller by taking $e = 1$, so $q = y$ and 
$| z - q w |^{2} = \lambda_{1}^{2} +  \lambda_{2}^{2} - |y|^{2}$. So we want to maximize 
$|y|^{2} = \lambda_{1}^{2} c^{2} + \lambda_{2}^{2} d^{2}$ subject to 
$c^{2} + d^{2} = 1$. Since $\lambda_{1} > \lambda_{2} \geq 0$ by assumption, 
we take $c = \pm 1$ and $d = 0$ so 
    \begin{equation}  \label{E:q=y.and.w=v1}
      p = q = y = \pm \lambda_{1} u_{1} \text{ and } w = \pm v_{1}^{T} .
    \end{equation} 
This choice, and the previous choice 
$r = \bar{x}$, minimizes $\bigl| x - (1^{n} r + p w) \bigr|$, with $p = q + s 1^{n}$. 

Now for row-wise perpendicularity. We have just seen that the closest line to $x$ is parametrized by $t \mapsto \bar{x} + t w = \bar{x} + t v_{1}^{T}$ and the closest vector of points on that line to $x$ is $1^{n} r + p w = 1^{n} \bar{x} + \lambda_{1} u_{1} v_{1}^{T}$. We show that  $x - ( 1^{n} \bar{x} + \lambda_{1} u_{1} v_{1}^{T} )$ is perpendicular to that line, i.e., to $w$. By \eqref{E:center.x}, \eqref{E:z.in.trms.of.v1.v2}, \eqref{E:q=y.and.w=v1}, 
    \begin{equation*}
      \bigl[ x - ( 1^{n} \bar{x} + \lambda_{1} u_{1} v_{1}^{T} ) \bigr] w^{T}
         = \bigl[ z - \lambda_{1} u_{1} v_{1}^{T} \bigr] w^{T}
           = \pm (\lambda_{2} u_{2} v_{2})^{T} v_{1} = 0^{n \times 1} ,
    \end{equation*}
as desired. 

(Yes, this generalizes.)
  \end{proof}

  \begin{proof}[Proof of \eqref{E:DR.mu(z).1st.version}]
$R_{\mu} : \Y \setminus \{0\} \to Y \setminus \{0\}$ is defined in \eqref{E:R.mu.defn}. Regard $\Y \setminus \{0\}$ as a subset of $\RR^{n \nvar}$. Let $Y \in \Y \setminus \{0\}$ and let $m := n \nvar$. Reshape $Y$ to be the column vector $z^{m \times 1}$ formed by concatenating the columns of $Y$ in order. Thus, 
if $(y^{j})^{n \times 1}$ is the $j^{th}$ column of $Y$ with transpose $y^{j T}$ 
($j=1, \ldots, \nvar$) we have 
$z^{m \times 1} = (y^{1T} , \ldots, y^{\nvar T})^{T} \neq 0$. 
Call this reshaping operation ``elongation''. Then $|z| = \|Y\|$. (See \eqref{E:matrix.norm}.) Identify each tangent space $T_{z}(\RR^{m})$ with $\RR^{m}$. 

For $z^{m \times 1} \neq 0$, define 
    \begin{equation}  \label{E:xi(z).and.nu(z).defns}
      \xi(z)^{m \times 1} := |z|^{-1} z \in \RR^{m} 
        \text{ and } \nu(z) := \mu \bigl( \xi(z) \bigr) \in \RR ,
    \end{equation} 
where $\mu$ is defined in \eqref{E:mu.as.in.D.mu.defn}. 
So $R_{\mu}(z) = \nu(z) \xi(z)^{m \times 1}$. Define 
$sprod(s, v^{m \times 1})^{m \times 1} = s v$ ($s \in \RR, v \in \RR^{m})$. Thus, 
$R_{\mu}(z) = sprod \bigl( \nu(z) , \xi(z) \bigr)$. The Jacobian matrix 
(Boothby \cite[p.\ 26]{wmB75}) of $sprod$ is 
$D \, sprod(s, v) = (v, s I_{m})^{m \times (m+1)}$, where $I_{m}$ is the $m \times m$ identity matrix. 

In the present case $s = \nu(z)$ and $v = \xi(z)$. 
Let $g(z) := \bigl( \nu(z), \xi(z)^{T} \bigr)^{T}$ so $R_{\mu} = sprod \circ g$. Then, by the chain rule (Boothby \cite[Theorem (2.3) p.\ 27]{wmB75}), the Jacobian matrix 
of $R_{\mu}$ is given by:
    \begin{align}  \label{E:DR.mu(z)}
      D R_{\mu}(z)^{m \times m} 
      &=  \bigl( D \, sprod \bigl[ g(z) \bigr] \bigr) D g(z) \notag \\
       &= \bigl( \xi(z) , \nu(z) I_{m} \bigr)^{m \times (m+1)}  
            \begin{pmatrix}
               \nabla \nu(z)^{1 \times m} \\
               D \xi(z)^{m \times m} 
            \end{pmatrix} \\
              &= \xi(z)^{m \times 1} \, \nabla \nu(z)^{1 \times m}
                + \nu(z) D \xi(z)^{m \times m} \notag .
    \end{align}

Consider $\xi$, defined in \eqref{E:xi(z).and.nu(z).defns}. Let $z_{i}$ be the $i^{th}$ element of $z$ ($i = 1, \ldots, m$). Notice that 
    \begin{equation*}
      \frac{\partial}{\partial z_{i}} \frac{z_{k}}{|z|} = 
        \begin{cases}
         - z_{i} z_{k}/|z|^{3}, &\text{ if } k \neq i, \\
          \bigl( |z|^{2} - z_{i}^{2} \bigr)/|z|^{3}, &\text{ if } k = i .
        \end{cases}
    \end{equation*}
Therefore,  
    \begin{equation}  \label{E:D.xi(z)}
      D \xi(z)^{m \times m} = - |z|^{-3} z z^{T} + |z|^{-1} I_{m} .
    \end{equation}
Hence, as one expects, 
    \begin{equation}  \label{Ez.perp.Dxi(z)}
      D \xi(z) z = - |z|^{-3} |z|^{2} z + |z|^{-1} z = 0 .
    \end{equation}

On the otherhand,  
    \begin{equation}  \label{E:D.xi.f.when.f.perp.z}
      \text{if } (f^{1 \times m})^{T} \perp z \text{ and } f \neq 0, \text{ then }
        D \xi(z) f^{T} = 0 + |z|^{-1} f^{T} \neq 0 . 
    \end{equation}
Let $f_{1}, \ldots, f_{m-1} \in \RR^{m}$ be $1 \times m$ row vectors orthonormal and orthogonal to $z$, which means orthogonal ot $\xi(z)$. 
Then the (multivariate) directional derivative of $\xi$ in the direction 
$f_{i}$ (i.e., $\tfrac{d}{dt} \xi (z + tf_{i}^{T}) \restriction_{t=0}$) is $|z|^{-1} f_{i}^{T}$ 
($i = 1, \ldots, m-1$). 
Let $f_{m}^{1 \times m} = \xi(z)^{T}$, so the derivative of $\xi$ in the direction $f_{m}$ is 0. 

Consider $\nu$, also defined in \eqref{E:xi(z).and.nu(z).defns}. We have 
    \begin{equation}  \label{E:nu.ast=mu.ast.circ.xi.ast}
      \nu_{\ast} = \mu_{\ast} \circ \xi_{\ast} ,
    \end{equation}
where $\nu_{\ast}$ is the differential of $\nu$, etc. Now, $\mu_{\ast}$, is only defined on the tangent bundle $T S^{m-1}$. Define $\nabla \mu \bigl[ \xi(z) \bigr]^{1 \times m}$ to be the row vector with the following property. Let $v^{1 \times m} \in T_{\xi(z)} S^{m-1}$, the tangent space to $S^{m-1}$ at $\xi(z)$. So $v$ is expressed as a vector in $\RR^{m}$. Then, at $\xi(z)$, the derivative of $\mu$ in the direction $v$ is 
$\nabla \mu \bigl[ \xi(z) \bigr] v^{T}$. 
WLOG we may take $\nabla \mu \bigl[ \xi(z) \bigr] \in T_{\xi(z)} S^{m-1}$.    

Let $F^{(m-1) \times m} = F(z)$ be the matrix whose $i^{th}$ row is $f_{i}$ 
($i = 1, \ldots, m-1$).  Thus, the rows of $F(z)$ constitute an orthonormal basis of the tangent space $T_{\xi(z)} S^{m-1} = z^{\perp}$. By \eqref{E:proj.mat.from.K}, this means $F^{T} F$ is the matrix of orthogonal projection $\RR^{m} \to T_{\xi(z)} S^{m-1}$. Since $\nabla \mu \bigl[ \xi(z) \bigr] \in T_{\xi(z)} S^{m-1}$, we thus have 
    \begin{equation}  \label{E:nabla.mu.FT.F}
      \nabla \mu \bigl[ \xi(z) \bigr] F^{T}F = \nabla \mu \bigl[ \xi(z) \bigr] .
    \end{equation}

From \eqref{E:D.xi.f.when.f.perp.z} and \eqref{Ez.perp.Dxi(z)}, we know that $D \xi(z) F^{T} =  |z|^{-1} F^{T}$ and $D \xi(z) f_{m}^{T} = 0$. Therefore, by  \eqref{E:nu.ast=mu.ast.circ.xi.ast} we have
    \begin{equation*}
      \nabla \nu(z)^{1 \times m} F^{T} 
        = \nabla \mu \bigl[ \xi(z) \bigr] \bigl( D\xi(z) F^{T} \bigr)
          = |z|^{-1} \nabla \mu \bigl[ \xi(z) \bigr] F^{T} 
    \end{equation*}
and $\nabla \nu(z)^{1 \times m} f_{m}^{T} = 0$. Let 
    \begin{equation*}
       F_{0}^{m \times m} =  
        \begin{pmatrix}
          F \\
          f_{m}
        \end{pmatrix}
\end{equation*}
Thus, $F_{0}$ is an orthonormal matrix: $F_{0}^{T} F_{0} = I_{m}$. And  
$\nabla \nu(z)^{1 \times m} F_{0}^{T} 
= |z|^{-1} \Bigl( \nabla \mu \bigl[ \xi(z) \bigr] F^{T}, 0 \Bigr)$. Therefore, by \eqref{E:nabla.mu.FT.F},  
    \begin{multline*}
      \nabla \nu(z)^{1 \times m} = \nabla \nu(z) F_{0}^{T} F_{0} \\
        = |z|^{-1} \Bigl( \nabla \mu \bigl[ \xi(z) \bigr] F^{T}, 0 \Bigr) F_{0} \notag  
          = |z|^{-1} \nabla \mu \bigl[ \xi(z) \bigr] F^{T} F 
            = |z|^{-1} \nabla \mu \bigl[ \xi(z) \bigr] .
    \end{multline*}
Substituting the preceding and \eqref{E:D.xi(z)} into \eqref{E:DR.mu(z)} we get: 
    \begin{equation}  \label{E:DR.mu.expanded}
      D R_{\mu}(z)^{m \times m} 
        = |z|^{-1} \xi(z) \nabla \mu \bigl[ \xi(z) \bigr] 
         + \nu(z) \bigl( - |z|^{-3} z z^{T} + |z|^{-1} I_{m} \bigr) .
    \end{equation}
Substituting $x := \xi(z)$ and $\mu(x) = \nu(z)$ (see \eqref{E:xi(z).and.nu(z).defns}), the preceding is \eqref{E:DR.mu(z).1st.version}. 
  \end{proof}
    
  \begin{proof}[Proof of lemma \ref{L:nontriv.homol.if.cov.by.bndl.map}]
 Let $w_{i}(\psi_{j})$ denote the $i^{th}$ Stiefel-Whitney class of $\psi_{j}$ 
($i = 1, 2, \ldots$; $j = 1,2$).
By assumption, $w_{s}(\psi_{1})$ is nontrivial. Therefore, by naturality of Stiefel-Whitney classes 
(Milnor and Stasheff \cite[Axiom 2, p.\ 37]{jwMjdS74})
$f^{\ast}$ is nontrivial in dimension $s$. Since the coefficient ring, $F := \mathbb{Z}/2$ is a field, we have by Munkres
\cite[Theorem 53.5, p.\ 325]{jrM84} that the following commutes and the rows are exact sequences.
   \[
      \begin{CD}
            0 @<<< \text{Hom}_{F} \,\bigl( H_{s}(B_{1}; F), F \bigr) @<<< H^{s}(B_{1}; F) 
              @<<< 0 \\
            & & @A{\text{Hom}_{F} \,(f_{\ast})}AA  @AA{f^{\ast}}A    \\
            0 @<<< \text{Hom}_{F} \,\bigl( H_{s}(B_{2}; F), F \bigr) @<<< H^{s}(B_{2}; F) 
              @<<< 0.
      \end{CD}
   \]
It follows that $f_{\ast}$ is nontrivial in dimension $s$.
  \end{proof}
  
  \begin{proof}[Proof of lemma \ref{L:Upsilon.mu:P1.to.Pf.is.imbedding}]
Since $P^{1}$ is compact and $\Upsilon_{\mu}$ injective (by \eqref{E:plane-fitting.Upsilon.mu.is.injective}), Boothby \cite[Theorem (5.7), p.\ 79]{wmB75} tells us that it suffices to show that 
$\Upsilon_{\mu \ast} : T_{\ell} P^{1} \to T_{\Upsilon_{\mu}(\ell)} \Pf^{k}$ is injective at each $\ell \in P^{1}$. 
(See Boothby \cite[Definition (4.3), p.\ 70]{wmB75}.)
By \eqref{E:Upsilon.Delta.mu.defn},  
and Boothby \cite[Theorem (1.2), p.\ 107]{wmB75}, 
    \begin{equation}  \label{E:Upsilon.mu.ast=R.mu.ast.circ.Upsilon.ast}
      \Upsilon_{\mu \ast} = R_{\mu \ast} \circ \Upsilon_{\ast} .
    \end{equation}

We have analyzed $R_{\mu \ast}$ in \eqref{E:DR.mu(z)} (a.k.a.\ \eqref{E:DR.mu(z).1st.version}). Next, we turn to $\Upsilon_{\ast}$. As remarked just after \eqref{E:P1.diffeom.to.S1}, $P^{1}$ can be locally parametrized by 
$s_{J} : \theta \mapsto \text{ span of } (\cos \theta, \sin \theta) \rangle$ 
for $\theta \in J$, where $J$ is any open interval of length $\pi$. Recall the definition \eqref{E:lambda(ell).defn} of $\lambda$. We analyze the map 
$\theta \mapsto \lambda \circ s_{J}(\theta)$. WLOG we may assume 
$v_{1} = (1, 0, 0, \ldots, 0)^{1 \times \nvar}$, 
$v_{2} = (0, 1, 0, \ldots, 0)^{1 \times \nvar}$, and $\zeta$ is spanned by the rows of
$(0^{(k-1) \times 2}, I_{k-1}, 0^{(k-1) \times (\nvar - k - 1)})^{(k-1) \times \nvar}$. 
Then $\lambda \circ s_{J}(\theta)$ is spanned by the rows of
    \begin{equation*}
      X := X(\theta)^{k \times \nvar} :=
        \begin{pmatrix}
          \cos \theta & \sin \theta & 0^{1 \times (k-1)} & 0^{1 \times (\nvar - k - 1)} \\
          0^{(k-1) \times 1} & 0^{(k-1) \times 1} & I_{k-1} & 0^{(k-1) \times (\nvar - k - 1)}
        \end{pmatrix} ,
        \qquad \theta \in J .
    \end{equation*} 

Let $\Pi$ be as in \eqref{E:defn.of.Pi(xi)}. 
Since the rows of $X$ are orthonormal, by \eqref{E:proj.mat.from.K}, the projection matrix 
onto $\zeta \circ s_{J}(\theta)$ is
    \begin{multline}  \label{E:Pi.lambda.circ.sJ}
      \Pi \bigl[ \lambda \circ s_{J}(\theta) \bigr]^{\nvar \times \nvar}= X^{T} X \\
        =
        \begin{pmatrix}
          \cos^{2} \theta & (\cos \theta)(\sin \theta) & 0^{1 \times (k-1)} 
            & 0^{1 \times (\nvar - k - 1)} \\
          (\cos \theta)(\sin \theta) & \sin^{2} \theta & 0^{1 \times (k-1)} 
            & 0^{1 \times (\nvar - k - 1)} \\
          0^{(k-1) \times 1} & 0^{(k-1) \times 1} & I_{k-1} & 0^{(k-1) \times (\nvar - k - 1)} \\
          0^{(\nvar - k - 1) \times 1} & 0^{(\nvar - k - 1) \times 1} 
            & 0^{(\nvar - k - 1) \times (k-1)} & 0^{(\nvar - k - 1) \times (\nvar - k - 1)}
        \end{pmatrix} .
    \end{multline}
Just before \eqref{E:plane.fitting.wT.Y=0} we learned that 
$\mbf{Y}^{n \times \nvar} \in \Y$ has full rank $\nvar$. Let the columns of $\mbf{Y}$ be denoted by $y^{1}, y^{2}, \ldots, y^{\nvar}$, all $n \times 1$. Recall \eqref{E:mcl.Y,defn}. 
Let $Y^{n \times \nvar} \in \Y \setminus \{0\}$ to be the matrix 
\eqref{E:Pi.lambda.circ.sJ}, we are interested in
    \begin{multline*}  
      Y := \Upsilon \bigl[ s_{J}(\theta) \bigr] 
        = \mbf{Y} \, \Pi \bigl[ \lambda \circ s_{J}(\theta) \bigr] \\
          = \Bigl( (\cos^{2} \theta) y^{1} + (\cos \theta \; \sin \theta) y^{2} ,
            (\cos \theta \; \sin \theta) y^{1} + (\sin^{2} \theta) y^{2} , 
              y^{3}, \ldots, y^{k+1}, 0^{n \times (\nvar-k-1)} \Bigr)^{n \times \nvar} .
    \end{multline*}

\emph{Claim:} The rank of this matrix is $k$. To see this, observe that the rank cannot exceed $k$ because, by \eqref{E:lambda(ell).imbedding}, 
$\lambda \circ s_{J}(\theta)$ is a $k$-plane. And since $\mbf{Y}$ has full rank, 
the $n$-vectors 
$(\cos \theta \; \sin \theta) y^{1} + (\sin^{2} \theta) y^{2}, y^{3}, \ldots, y^{k+1}$ are linearly independent. This proves the claim.  

$m := n \nvar$. As in the proof of \eqref{E:DR.mu(z).1st.version} above, reshape $Y$ to be be the $m$-dimensional column vector $z^{m \times 1}$ formed by concatenating the columns of $Y$ in order. 
Call this reshaping operation ``elongation''. Then $|z| = \|Y\| > 0$. (See \eqref{E:matrix.norm}.) The first two columns of $Y$ are linearly dependent: If we multiply the first column by $\sin \theta$ and the second by $\cos \theta$ we get identical column vectors. Denote the transpose of $y^{j}$ by $y^{j T}$. Thus, 
    \begin{multline}  \label{E:z.elongation.of.Y.Pi}
     z^{m \times 1} = \Bigl( (\cos^{2} \theta) y^{1 T} + (\cos \theta \; \sin \theta) y^{2 T} ,
           (\cos \theta \; \sin \theta) y^{1 T} + (\sin^{2} \theta) y^{2 T} , \\
             y^{3 T}, \ldots, y^{(k+1) T}, 
               0^{1 \times (\nvar (n-k-1)} \Bigr)^{T} .
    \end{multline}

By \eqref{E:Pi.lambda.circ.sJ}, 
$\tfrac{d}{d \theta}  \Pi \bigl[ \lambda \circ s_{J}(\theta) \bigr]$ is the 
$\nvar \times \nvar$ matrix which is all 0 except in the first two rows and columns in which are
    \begin{equation*}
        \begin{pmatrix}
           -2 (\cos \theta)(\sin \theta) & - \sin^{2} \theta + \cos^{2} \theta \\
           - \sin^{2} \theta + \cos^{2} \theta & 2 (\cos \theta)(\sin \theta)
        \end{pmatrix} =
        \begin{pmatrix}
           - \sin 2 \theta & \cos 2 \theta \\
           \cos 2 \theta & \sin 2 \theta
        \end{pmatrix} .
    \end{equation*}
Therefore, at $\ell = \ell(\theta)$, 
    \begin{multline*}
      \Upsilon_{\ast}(\ell) 
        = \frac{d}{d \theta} \mbf{Y} \,  \Pi \bigl[ \lambda \circ s_{J}(\theta) \bigr] =
           \mbf{Y} \frac{d}{d \theta}  \Pi \bigl[ \lambda \circ s_{J}(\theta) \bigr] \\
             = \bigl( - (\sin 2 \theta) y^{1} + (\cos 2 \theta) y^{2} , \;
               (\cos 2 \theta) y^{1} + (\sin 2 \theta) y^{2} , 
                 \; 0^{n \times (\nvar-2)} \bigr)^{n \times \nvar} .
    \end{multline*}
After elongation the final expression in the preceding becomes
    \begin{multline}  \label{E:w=Upsilon.ast}
     w^{m \times 1} := \Upsilon_{\ast}(\ell) \\
       = \bigl( - (\sin 2 \theta) y^{1 T} + (\cos 2 \theta) y^{2 T} , \;
            (\cos 2 \theta) y^{1 T} + (\sin 2 \theta) y^{2 T} , 
              \; 0^{1 \times n(\nvar-2)} \bigr)^{1 \times m} .
    \end{multline}
Since $\mbf{Y}$ has full rank $\nvar$, $y^{1}$ and $y^{2}$ are linearly independent, 
so $w \neq 0$.

\emph{Claim:} $w$ and $z$ are linearly independent. Suppose not. Then there exist 
$a, b \in \RR$ not both 0 s.t.\ $a z + b w = 0^{m \times 1}$. We may reverse the elongation, take the first two columns, and by \eqref{E:z.elongation.of.Y.Pi} and \eqref{E:w=Upsilon.ast} conclude 
    \begin{align*}
      0^{n \times 2} &= a \bigl[ (\cos^{2} \theta) y^{1} + (\cos \theta \; \sin \theta) y^{2} ,
               (\cos \theta \; \sin \theta) y^{1} + (\sin^{2} \theta) y^{2} \bigr]^{n \times 2} \\
        &\qquad + b \bigl[ - (\sin 2 \theta) y^{1} + (\cos 2 \theta) y^{2} , \;
                (\cos 2 \theta) y^{1} + (\sin 2 \theta) y^{2} \bigr]^{n \times 2} \\
        &=  a\, (y^{1} , y^{2})      
            \begin{pmatrix}
             \cos^{2} \theta & \cos \theta \; \sin \theta \\
             \cos \theta \; \sin \theta & \sin^{2} \theta
            \end{pmatrix}
           + b \, (y^{1} , y^{2})      
            \begin{pmatrix}
              -\sin 2 \theta  & \cos 2 \theta \\
              \cos 2 \theta & \sin 2 \theta
            \end{pmatrix} \\
       &= (y^{1} , y^{2}) \left[ a 
             \begin{pmatrix}
               \cos^{2} \theta & \cos \theta \; \sin \theta \\
               \cos \theta \; \sin \theta & \sin^{2} \theta
            \end{pmatrix} + b 
            \begin{pmatrix}
              -\sin 2 \theta  & \cos 2 \theta \\
              \cos 2 \theta & \sin 2 \theta
            \end{pmatrix} \right] .
    \end{align*}
Since $\mbf{Y}$ has full rank, $y^{1}$ and $y^{2}$ are linearly independent. It follows that the expression enclosed in the brackets is 0. That implies 
    \begin{align*}
      0 &= (-\sin \theta, \cos \theta)
        \left[ a 
            \begin{pmatrix}
         \cos^{2} \theta & \cos \theta \; \sin \theta \\
         \cos \theta \; \sin \theta & \sin^{2} \theta
        \end{pmatrix} \right. \\
        &\qquad \left. + b 
            \begin{pmatrix}
              -\sin 2 \theta  & \cos 2 \theta \\
              \cos 2 \theta & \sin 2 \theta
            \end{pmatrix} \right] \\
        &=  a (- \sin \theta \cos^{2} \theta + \sin \theta \cos^{2} \theta ,
            - \cos \theta \sin^{2} \theta + \cos \theta \sin^{2} \theta) \\
        &\qquad    + b (-\sin \theta, \cos \theta) 
        \begin{pmatrix}
        -\sin 2 \theta  & \cos 2 \theta \\
        \cos 2 \theta & \sin 2 \theta
        \end{pmatrix} \\
       &= 0 + b (-\sin \theta, \cos \theta) 
        \begin{pmatrix}
        -\sin 2 \theta  & \cos 2 \theta \\
        \cos 2 \theta & \sin 2 \theta
        \end{pmatrix} .
    \end{align*}
The determinant of the last matrix, 
    $\left(
      \begin{smallmatrix}
            -\sin 2 \theta  & \cos 2 \theta \\
            \cos 2 \theta & \sin 2 \theta
      \end{smallmatrix}
    \right)$,
is -1, so that matrix has full rank. So $b$ multiplies a nonzero vector. Thus, $b = 0$. The first matrix, 
    $\left(
      \begin{smallmatrix}
         \cos^{2} \theta & \cos \theta \; \sin \theta \\
         \cos \theta \; \sin \theta & \sin^{2} \theta
      \end{smallmatrix}
    \right)$,
does not have full rank, but it is not 0. Hence, $a = 0$ as well. This contradicts the assumption that $a$ and $b$ are not both 0 and proves the claim that $w$ and $z$ are linearly independent. In particular, neither $w$ nor $z$ is 0. It follows that 
$w' := w - |z|^{-2} (w \cdot z) z \neq 0$. We have $w = w' + |z|^{-2} (w \cdot z) z$ 
and $z \cdot w' = 0$.

For $z^{m \times 1} \neq 0$, define $\xi(z)^{m \times 1}$ and $\nu(z)$ as in \eqref{E:xi(z).and.nu(z).defns}. So $R_{\mu}(z) = \nu(z) \xi(z)^{m \times 1}$. By \eqref{E:D.xi(z)} and \eqref{Ez.perp.Dxi(z)}, 
    \begin{equation}  \label{E:Dxi(z).w.neq.0}
        D\xi(z) w = D\xi(z) w' = |z|^{-1} w' \neq 0 .
    \end{equation}
 
Therefore, by \eqref{E:Upsilon.mu.ast=R.mu.ast.circ.Upsilon.ast}, \eqref{E:DR.mu(z)}, \eqref{E:w=Upsilon.ast}, and \eqref{E:Dxi(z).w.neq.0}, 
    \begin{align}  \label{E:Oops!.mu.ast}
       \Upsilon_{\mu \ast}(\ell) 
         &= \Bigl( \xi(z)^{m \times 1} \, \nabla \nu(z)^{1 \times m} 
            + \nu(z) D \xi(z)^{m \times m}  \Bigr) w^{m \times 1} \notag \\
         &= \xi(z) \, \nabla \nu(z) w 
            + \nu(z) D \xi(z) w \\
         &= \xi(z) \, \nabla \nu(z) w 
            + \nu(z) |z|^{-1} w' \notag .
    \end{align}

Now, by \eqref{E:xi(z).and.nu(z).defns}, $\xi(z) \propto z$, so $\xi(z) \perp w'$. Moreover,  
$\nu(z) |z|^{-1} \neq 0$. Therefore, if $\xi(z) \, \nabla \nu(z) w \neq 0$, it is linearly independent of the non-zero vector $\nu(z) |z|^{-1} w'$. Hence, the last vector in \eqref{E:Oops!.mu.ast} is non-zero. I.e.,  
$\Upsilon_{\mu \ast}(\ell) \neq 0$. So the rank of $\Upsilon_{\mu \ast}(\ell)$ is at least 1. But the domain of $\Upsilon_{\mu}$ is $P^{1}$, a one-dimensional manifold.  
Therefore $\Upsilon_{\mu}$ is an immersion (Boothby \cite[Definition (4.3), p.\ 70]{wmB75}). By \eqref{E:plane-fitting.Upsilon.mu.is.injective} it is injective. But $P^{1}$ is compact, so by Boothby \cite[Theorem (5.7), p.\ 79]{wmB75}, $\Upsilon_{\mu}$ is an imbedding.  
  \end{proof}

  \begin{proof}[Proof of lemma \ref{L:collinearity.and.w.1xn}]
 First, suppose $w^{n \times 1}$ satisfies 
$rank \, (X - 1^{n} w^{T} X) < k$. Then there exists 
$a^{k \times 1} \neq 0$ s.t.\ $(X - 1^{n} w^{T} X) a = 0^{n \times 1}$. 
I.e., $Xa = 1^{n} w^{T} X a$. Let $f_{1}^{n \times 1} := (1, 0, \ldots, 0)^{T}$. Then, 
	\begin{multline*}
	  \bigl[ (I_{n} - 1^{n} f_{1}^{T} )X \bigr] a = (I_{n} - 1^{n} f_{1}^{T} ) (Xa)
	    = (I_{n} - 1^{n} f_{1}^{T} ) (1^{n} w^{T} X a) \\
	    = 1^{n} w^{T} X a - 1^{n} (f_{1}^{T} 1^{n}) w^{T} X a 
	      = 1^{n} w^{T} X a - 1^{n} w^{T} X a = 0.
	\end{multline*}
I.e., $Y$ is collinear.

Conversely, suppose $Y$ is collinear and $w^{n \times 1}$ satisfies $w^{T} 1^{n} = 1$. Then there exists 
$a^{k \times 1} \neq 0$ s.t.\ $(I_{n} - 1^{n} f_{1}^{T} ) X a = 0^{n \times 1}$. Thus,
$X a = 1^{n} f_{1}^{T} X a$ and 
	\begin{equation*}
	  (X - 1^{n} w^{T} X) a = (I_{n} - 1^{n} w^{T} ) 1^{n} f_{1}^{T} X a 
	    = 1^{n} f_{1}^{T} X a - 1^{n} (w^{T} 1^{n}) f_{1}^{T} X a = 0 .
	\end{equation*}
The lemma is proved.
  \end{proof}

  \begin{proof}[Proof of lemma \ref{L:Y.collnr.iff.rank.X1.<.k+1}] 
Let $n \geq \nvar = k+m$. (See \eqref{E:n.>.k+m}.) First, suppose $Y$ is collinear. We show that $\text{rank} \, X_{1} < k+1$. For suppose not, then for every $v^{1 \times k}$ there exists $c^{1 \times n}$ s.t.\ 
	\begin{equation*} \label{E:0.v.from.X1}
		(0^{1 \times 1}, v) = c X_{1}. 
	\end{equation*}
It follows that $c$ is perpendicular to the first column of $X_{1}$. I.e., $c 1^{n} = 0$, where $1^{n}$ is the $n$-dimensional column vector consisting only of $1$'s. Let $W^{(n-1) \times k}$ be the matrix whose $i^{th}$ row is $x_{i+1} - x_{1}$ and write $c = (c_{1}^{1 \times 1}, c_{2}^{1 \times (n-1)})$. Notice that $( -1^{n-1},  I_{n-1} ) 1^{n} = 0^{(n-1) \times 1}$ 
and $( -1^{n-1},  I_{n-1} ) X = W$. Thus, we have 
	\begin{multline*}
		(0^{1 \times 1},v) = c \bigl[ (-1^{n}, \, 0^{n \times (n-1)} ) + I_{n} \bigr] X_{1} = (c_{1}, c_{2})
			\begin{pmatrix}
				0^{1 \times 1} & 0^{1 \times (n-1)} \\
				-1^{n-1} & I_{n-1}
			\end{pmatrix}
		   X_{1}  \\
		      = c_{2} ( -1^{n-1},  I_{n-1} )(1^{n}  ,  X) = c_{2}(0^{(n-1) \times 1}, \; W).
	\end{multline*}
Hence, $c_{2} W = v$. But $v^{1 \times k}$ is arbitrary. Hence, the rows of $W$ span $\RR^{k}$. By definition \ref{D:collinearity} of collinearity this means $Y = (X, Z)$ is not collinear. Contradiction.

Next we prove that, conversely, if $Y$ is not collinear, then the matrix $X_{1}$ has full rank $k + 1$. For suppose not. Then there exists $c^{\nvar \times 1} = (c_{0}, c_{1}, \ldots, c_{k})^{T} \in \RR^{k+1}$ s.t.\ $c \neq 0$ but $X_{1} c = 0^{n \times 1}$. Let $c' :=  (c_{1}, \ldots, c_{k})^{T} \in \RR^{k}$. I.e., $c'$ is just $c$ with the first coordinate dropped. We have $0 = X_{1} c = c_{0} 1^{n} + X c'$. 
Hence $c \neq 0$ implies $c' \neq 0$ and moreover,
	\begin{equation*} \label{E:xi.c'.const}
		\text{For } i \in \NN_{n}, \text{ we have } x_{i} c' = -c_{0}.
	\end{equation*} 
Thus, for $i = 2, \ldots, n$, we have $(x_{i} - x_{1}) c' = 0$. I.e.,
	\[
		W c' = 0,
	\]
where $W$ is defined in the last paragraph. But $Y$ is not collinear so $x_{2} - x_{1}, \ldots, x_{n}- x_{1}$ spans $\RR^{k}$. Hence, there exists $a^{1 \times n} \in \RR^{n}$ s.t.\ $a W = (c')^{T}$. Thus, 
	\[
		0 < |c'|^{2} = a W c' = 0.
	\]
This contradiction proves the lemma. 
  \end{proof}

\begin{proof}[Proof of proposition \ref{P:collin.data.are.sings.of.LS}]
We will continue the custom of using superscripts to indicate the dimension of matrices. 
Let $Y^{n \times \nvar} = (X^{n \times k}, Z^{n \times m}) \in \Y$. We are interested in the LS planes for $Y$. By \eqref{E:subtracting.const.doesn't.change.LS.beta}, we may assume
    \begin{equation*}
         X \text{ is mean-centered.}
    \end{equation*}
Note that mean-centering a matrix is a continuous operation.

First, suppose
    \begin{equation*}
         Y \text{ is not collinear.}
    \end{equation*}
Then, by \eqref{E:Y.not.collin.then.Phi(Y).unique}, $\Phi_{LS}(Y)$ exists uniquely. We show that  $Y$ is not a singularity 
of $\Phi := \Phi_{LS}$ w.r.t.\ $\Y'$. By lemma \ref{L:Y.collnr.iff.rank.X1.<.k+1}, we have $rank \, X_{1} = k+1$ and the LS estimates, $(\hat{\alpha}^{1 \times m}, \, \hat{\beta}^{k \times m})$, for $Y$ are given by \eqref{E:LS.estimate.formula}. By lemma \ref{L:collinearity.and.w.1xn}, $X$, which we are assuming is mean-centered, 
has rank $k$. 

Let $(Z')^{n \times m}$ be arbitrary. By lemma \ref{L:rank.lwr.semicont}, there is a neighborhood, 
$\X$, 
of $X^{n \times k}$ s.t.\ if $X' \in \X$ then $rank \, X'_{1} = k+1$. Hence, by lemma \ref{L:Y.collnr.iff.rank.X1.<.k+1} 
this means $Y' := (X', Z') \in \Y'$. Moreover, the LS estimates, call them $\bigl( (\alpha',^{1 \times 1},  (\beta')^{k \times m} \bigr)$, are unique and given by \eqref{E:LS.estimate.formula}. 
Thus, $\bigl( (\alpha',^{1 \times 1},  (\beta')^{k \times m} \bigr)$ is continuous on $\Y'$. 
Let $e_{i}^{1 \times k}$ ($i = 1, \ldots, k$) be a basis for $\RR^{k}$. By \eqref{E:LS.plane}, 
the matrix $(B')^{k \times \nvar}$ ($\nvar = k + m$) whose $i^{th}$ row the vector
$(e_{i}, e_{i} \beta')$ ($i = 1, \ldots, k$) satisfies \ $\rho(B') = \Phi_{LS}(Y')$. As $Y' \to Y$, $B'$ converges the matrix $B$ whose $i^{th}$ row is 
to $(e_{i}, e_{i} \hat{\beta})$ ($i = 1, \ldots, k$). But $\rho(B) = \Phi_{LS}(Y)$. By \eqref{E:convergence.in.Grassmann}, it follows that $\Phi_{LS}(Y') \to \Phi_{LS}(Y)$ 
in $G(k, \nvar)$. 

Next, suppose 
    \begin{equation*}
     Y = (X,Z) \text{ is collinear}.
    \end{equation*}
Then, by lemma \ref{L:Y.collnr.iff.rank.X1.<.k+1} again, $X_{1}$ has rank $k'+1 < k+1$. 
Since $X$ is mean-centered and so whose columns are orthogonal to $1^{n}$, we have,
    \begin{equation} \label{E:rank.X.=.k'}
	  rank \, X = k' < k. \text{ Let } \ell := k-k'.   
    \end{equation}
By \eqref{E:X.not.full.rank.then.beta.not.unique}, the LS regression of $Z$ on $X$ is not (uniquely) defined. We will show that in this case $Y$ is a singularity of $\Phi$ w.r.t.\ 
$\Y'$.

Let $W^{n \times \ell}$ be a matrix whose columns are orthonormal and orthogonal to the column spaces of $X_{1}$ and $Z$. (This is possible since, by \eqref{E:n.>.k+m}, 
$n \geq \nvar+1 = k+1+m > k' + 1 + m$, so $n - (k' + 1 + m) \geq k-k' = \ell$.) In particular,
	\begin{equation} \label{E:W.is.mean.centered}
		1_{n} \, W = 0.
	\end{equation}
 For $j = 1, 2, \ldots$ let $I_{j}$ be the $j \times j$ identity matrix. Then summing up,
	\begin{equation} \label{E:W.X1.Z} 
		W^{T} (X_{1} , Z) = 0, \; W \text{ has dimensions } n \times \ell, 
		    \text{ and } W^{T} W = I_{\ell}.
	\end{equation}
	
 Let $A^{\ell \times k}$ have orthonormal rows orthogonal to $ \rho (X)$, the row space of $X$, 
 so $\rho(A)^{\perp} = \rho(X)$, where $\rho(A)^{\perp}$ is the orthogonal complement 
 of $\rho(A)$ in $\RR^{k}$. Then we have,
	\begin{equation} \label{E:A.and.X}
		X A^{T} = 0, \quad A \text{ is } \ell \times k, \quad A A^{T} = I_{\ell}.
	\end{equation}
 If $\epsilon > 0$ let 
	\begin{equation}  \label{E:X.eps.defn}
		X_{\epsilon} := X + \epsilon W A,
	\end{equation} 
so $X_{\epsilon} \to X$ as $\epsilon \to 0$. By assumption on $X$ and \eqref{E:W.is.mean.centered}, $X_{\epsilon}$ has mean-centered columns. 

\emph{Claim:} 
	\begin{equation}  \label{E:X.eps.has.rank.k}
		X_{\epsilon} \text{ has rank } k. 
	\end{equation}
For otherwise, there exists $a^{k \times 1} \neq 0$ s.t.\ $X_{\epsilon} a = 0$. Thus, by \eqref{E:W.X1.Z},
	\begin{equation*}
		0 = W^{T} X_{\epsilon} a = W^{T} X a + \epsilon (W^{T} W) A a = \epsilon A a.
	\end{equation*}
Therefore, $a^{T} \perp \rho(A)$. This means that $a^{T} \in \rho(X)$. Since $rank \, X = k' < k$, we may assume WLOG that the first $k'$ rows of $X$ are linearly independent. Let $X(1)^{k' \times k}$ and $X(2)^{(n-k') \times k}$ be the matrices consisting of the first $k'$ and last $n-k'$ rows of $X$, resp. Thus, 
	\begin{equation*}
		X =
			\begin{pmatrix}
				X(1) \\
				X(2)
			\end{pmatrix} .
	\end{equation*}
Now, $rank \, X(1) = k'$ so $X(1) X(1)^{T}$ is a $k' \times k'$ matrix of full rank $k'$. Since $a^{T} \in \rho(X) \setminus \{0\}$, there exists $b^{k' \times 1} \neq 0$ s.t.\ $a = X(1)^{T} b$. Thus, since $a^{T} \perp \rho(A)$, 
	\begin{equation*}
		0 = X_{\epsilon} a = X a + \epsilon W A a = X \, X(1)^{T} b =
			\begin{pmatrix}
				X(1) X(1)^{T} b \\
				X(2) X(1)^{T} b
			\end{pmatrix} .
	\end{equation*}
We conclude that $X(1) X(1)^{T} b = 0^{k' \times 1}$ so $b = 0$. This means $a = 0$, a contradiction that proves the claim that $rank \, X_{\epsilon} = k$. By \eqref{E:W.X1.Z} and the fact that $X$ has mean-centered columns, $1_{n} X_{\epsilon} = 0$. 
Therefore, $X_{\epsilon,1} := (1^{n}, X_{\epsilon})$ has rank $k+1$. 

Thus, by \eqref{E:LS.estimate.formula}, there is a unique LS regression of $Z$ on $X_{\epsilon}$. Let $\hat{\alpha}^{1\times m}$ and $\hat{\beta}^{k \times m}$ be LS estimates for the regression of $Z$ on $X$ ($X$, not $X_{\epsilon}$). (They exist but are not unique.)
Equation \eqref{E:normal.eqns.Z} and the fact that $X$ is mean-centered imply that $(\hat{\alpha}, \hat{\beta})$ can be any solution of the following.
	\begin{equation}  \label{E:normal.eqns.spelled.out}
			\begin{pmatrix}
				n \hat{\alpha} \\
				X^{T} \, X \, \hat{\beta}
			\end{pmatrix}
			=
			\begin{pmatrix}
				n & 0^{1 \times k} \\
				0 & X^{T} \, X
			\end{pmatrix}
			\begin{pmatrix}
				\hat{\alpha} \\
				\hat{\beta} 
			\end{pmatrix}
		=  X_{1}^{T} X_{1}
			\begin{pmatrix}
				\hat{\alpha} \\
				\hat{\beta} 
			\end{pmatrix}
		    = X_{1}^{T} Z = 
			\begin{pmatrix}
				1_{n} \,  Z \\
				X^{T} \, Z
			\end{pmatrix}^{(k+1) \times m}.
	\end{equation}
It follows that
	\begin{equation}  \label{E:alpha.hat.is.mean.of.rows}
		\hat{\alpha} = n^{-1} \, 1_{n} \,  Z.
	\end{equation}
I.e., $\hat{\alpha}$ is just the mean of the rows of $Z$. From \eqref{E:normal.eqns.spelled.out}, we may assume that  
	\begin{equation} \label{E:rows.of.beta.in.rho.X}
		\text{The columns of } \hat{\beta} \text{ lie in } \rho (X).
	\end{equation}

Let $V_{1}$ and $V_{2}$ be $\ell \times m$ (real) matrices with 
	\begin{equation}  \label{E:V1.V2.orthog}
		 V_{1}^{T} V_{2} = 0^{m \times m}.
	\end{equation}
(It is possible that $V_{1}$ or $V_{2}$, or both, is an all zero matrix.) Let
	\begin{equation}  \label{E:Z.eps.defn}
		Z_{\epsilon} := Z + \epsilon^{1/2} W V_{1}
		          + \epsilon W V_{2},
	\end{equation}
so $Z_{\epsilon} \to Z$ as $\epsilon \downarrow 0$. Let
	\begin{equation}  \label{E:B.eps.defn}
		B_{\epsilon}^{\ell \times m} 
		          := W^{T}  Z_{\epsilon}. 
	\end{equation}
Then, by \eqref{E:W.X1.Z}, 
	\begin{multline}  \label{E:expand.B.eps}
	B_{\epsilon} = W^{T} (Z + \epsilon^{1/2} W V_{1} + \epsilon W V_{2})  \\
		= W^{T} \, Z + \epsilon^{1/2} \, W^{T} \; WV_{1}+ \epsilon \, W^{T} \; WV_{2} 
		        = \epsilon^{1/2} V_{1} + \epsilon V_{2}.
	\end{multline}
Thus, if $V_{1}$ is nonzero then $\epsilon^{-1}  B_{\epsilon}$  blows up as $\epsilon  \downarrow  0$. Now let 
	\begin{equation*}
		Y_{\epsilon} := (X_{\epsilon},  Z_{\epsilon}),
	\end{equation*}

Observe that 
    \begin{equation*}
      Y_{\epsilon}  \to Y \text{ as } \epsilon  \downarrow  0 . 
    \end{equation*}
Since $X_{\epsilon}$ is mean-centered and has, by \eqref{E:X.eps.has.rank.k}, rank $k$, by lemma \ref{L:collinearity.and.w.1xn}, 
    \begin{equation*}
      Y_{\epsilon} \text{ is not collinear. I.e., } Y_{\epsilon} \in \mcl{Y}'.
    \end{equation*}
But it is not necessarily the case that $Y_{\epsilon} \in \Pf^{k}$.\footnote{Here is an example. Take $n := 4$, $k := 2$, and $m := 1$. Let $X^{4 \times 2}$ have $(1,0)$ as its first row, $(0,0)$ as its second and third rows, and $(-1,0)$ as its last row. So $X$ is mean-centered and has rank 1.
Thus, $\ell = 1$. Let $(Z^{4 \times 1})^{T} := (0, 1, 0, 0)$, so, by lemma \ref{L:collinearity.and.w.1xn}, 
$Y := (X,Z)$ is collinear. 
Let $(W^{4 \times 1})^{T} := 2^{-1/2} (0, -1, 1, 0)$, and $A^{1 \times 2} = (0,1)$. 
Take $V_{1} = V_{2} = 0^{1 \times 1}$ so $Z_{\epsilon} = Z$. Then 
    \begin{equation*}
    	Y_{\epsilon} = 
    		\begin{pmatrix}
    			1 & 0 & 0 \\
			0 & -\epsilon/\sqrt{2} & 1 \\
    			0 & \epsilon/\sqrt{2} & 0 \\
			-1 & 0 & 0 
    		\end{pmatrix} .
    \end{equation*}
Let $w = (0,1,0,0)^{T}$. Then $Y_{\epsilon} - 1^{n} w^{T} Y_{\epsilon}$ has rank 3. Therefore, by \eqref{E:when.is.Y.in.Pk}, $Y_{\epsilon} \notin \Pf$.} 
Let 
	\begin{equation}   \label{E:gamma.eps.defn}
		\gamma_{\epsilon}^{k \times m} := \epsilon^{-1} \, A^{T} \, B_{\epsilon} 
		     = \epsilon^{-1/2} A^{T} \, V_{1} + A^{T} V_{2}
	\end{equation} 
and
	\begin{equation}   \label{E:beta.hat.eps.defn}
		\hat{\beta}_{\epsilon} = \hat{\beta} + \gamma_{\epsilon}.
	\end{equation}

We have, by \eqref{E:W.X1.Z}, 
	\begin{equation}  \label{E:X1.X1.eps}
		X_{\epsilon}^{T} \; X_{\epsilon} 
		     = ( X^{T} \, X + \epsilon^{2} \, A^{T} \, A )^{k \times k}.
	\end{equation}
Moreover, by \eqref{E:gamma.eps.defn}, \eqref{E:A.and.X}, and \eqref{E:expand.B.eps}, 
	\begin{equation}  \label{E:expand.eps2.AT.A.gam}
		\epsilon^{2} \, A^{T} \; A \, \gamma_{\epsilon} 
		  = \epsilon A^{T} \, A A^{T} \, B_{\epsilon} = \epsilon A^{T} \, B_{\epsilon}
	             = \epsilon^{3/2} A^{T} \; V_{1} + \epsilon^{2} A^{T} \; V_{2}.
	\end{equation}

We know that the ``normal equations'', \eqref{E:normal.eqns.Z}, hold for $\hat{\alpha}$ and 
$\hat{\beta}$. \emph{Claim:} the normal equations hold in the ``$\epsilon$ world'', too. 
Let $X_{\epsilon, 1} := (1^{n}  \;  X_{\epsilon})$. Observe that since $X_{\epsilon}$, as in \eqref{E:normal.eqns.spelled.out} we have,
	\begin{equation}  \label{E:LHS.normal.eqns.epsilon}
	        X_{\epsilon, 1}^{T}   X_{\epsilon, 1} 				
	        			\begin{pmatrix}
					\hat{\alpha}^{T} \\
					\hat{\beta}_{\epsilon}
				\end{pmatrix}
	              = 
				\begin{pmatrix}
					n & 0^{1 \times k} \\
					0 & X_{\epsilon}^{T} \, X_{\epsilon}
				\end{pmatrix}
				\begin{pmatrix}
					\hat{\alpha} \\
					\hat{\beta}_{\epsilon} 
				\end{pmatrix}
		     =
				\begin{pmatrix}
					n \hat{\alpha} \\
					X_{\epsilon}^{T} \, X_{\epsilon} \hat{\beta}_{\epsilon} 
				\end{pmatrix}
	\end{equation}
and 
	\begin{equation}   \label{E:normal.eqns.epsilon.RHS}
	 	X_{\epsilon,1}^{T}   Z_{\epsilon} =
				\begin{pmatrix}
					1_{n} \,  Z_{\epsilon} \\
					X_{\epsilon}^{T} \, Z_{\epsilon}
				\end{pmatrix}^{(k+1) \times m}  .
	\end{equation}

Hence, combining \eqref{E:LHS.normal.eqns.epsilon} and \eqref{E:normal.eqns.epsilon.RHS}, to prove the claim it suffices to show  
	\begin{equation}  \label{E:beta.part.of.normal.eqns}
		n \hat{\alpha} = 1_{n} \,  Z_{\epsilon}
		   \text{ and } X_{\epsilon}^{T} \, X_{\epsilon} \; \hat{\beta}_{\epsilon}  
		     = X_{\epsilon}^{T} \, Z_{\epsilon}.
	\end{equation}
By \eqref{E:alpha.hat.is.mean.of.rows}, \eqref{E:Z.eps.defn}, and \eqref{E:W.is.mean.centered}, 
the $\hat{\alpha}$ equation holds.

It is easy to see, using \eqref{E:W.X1.Z}, that the following holds.
	\begin{equation}  \label{E:XepsT.Zeps}
		X_{\epsilon}^{T}   Z_{\epsilon} 
		  = X^{T}   Z  +  \epsilon^{3/2} A^{T} \; V_{1} + \epsilon^{2} A^{T} \; V_{2}.
	\end{equation}

On the other hand, by \eqref{E:X1.X1.eps} and \eqref{E:beta.hat.eps.defn}, 
	\begin{equation*}
		X_{\epsilon}^{T} \, X_{\epsilon} \; \hat{\beta}_{\epsilon}  
		=  (X^{T} \, X + \epsilon^{2} \, A^{T} \; A)(\hat{\beta} + \gamma_{\epsilon}).
	\end{equation*}
But, by \eqref{E:gamma.eps.defn}, \eqref{E:A.and.X}, \eqref{E:rows.of.beta.in.rho.X}, 
\eqref{E:expand.eps2.AT.A.gam}, and \eqref{E:normal.eqns.spelled.out}, 
	\begin{align*}
		(X^{T} \, X + \epsilon^{2} \, A^{T} \; A)(\hat{\beta} + \gamma_{\epsilon}) 
			&= (X^{T} \, X) \hat{\beta} + \epsilon^{2} A^{T} \; A \, \gamma_{\epsilon} \\
			&= (X^{T} \, X) \hat{\beta} + \epsilon^{3/2} A^{T} \; V_{1} 
			 		+ \epsilon^{2} A^{T} \; V_{2} \\
			&= X^{T} \, Z  + \epsilon^{3/2} A^{T} \; V_{1} 
			 		+ \epsilon^{2} A^{T} \; V_{2}.
	\end{align*}
Comparing this to \eqref{E:XepsT.Zeps}, \eqref{E:beta.part.of.normal.eqns} is proved. The claim that
$(\hat{\alpha}^{T}, \hat{\beta}_{\epsilon}^{T} )^{T}$ is a solution to the normal equations for $Y_{\epsilon}$ follows. By \eqref{E:X.eps.has.rank.k} and the fact that $X_{\epsilon}$ has mean-centered columns, we see that $X_{\epsilon,1}^{T}$ has full rank. Therefore, by \eqref{E:LS.estimate.formula} we have that $(\hat{\alpha}^{T}, \hat{\beta}_{\epsilon}^{T} )^{T}$ is the unique LS estimate of the regression of $Z_{\epsilon}$ on $X_{\epsilon}$.

Now, by \eqref{E:beta.hat.eps.defn} and \eqref{E:gamma.eps.defn}, 
	\begin{equation}  \label{E:beta.eps.revealed}
		\hat{\beta}_{\epsilon} = \hat{\beta} + \epsilon^{-1/2} A^{T} \, V_{1} + A^{T} V_{2}.
	\end{equation}
By \eqref{E:rows.of.beta.in.rho.X} and \eqref{E:A.and.X}, the column space of $\hat{\beta}$ and column space 
of $\epsilon^{-1/2} A^{T} \, V_{i}$ ($i = 1, 2$) are orthogonal. 

Let 
	\begin{equation} \label{E:V1T.perp.omega}
		\omega  := \rho(V_{1}^{T})^{\perp} \subset \RR^{\ell} 
		  \text{ so } \rho(V_{1}^{T}) \perp \omega.
	\end{equation}
I.e., $\omega$ is the orthogonal complement of $ \rho(V_{1}^{T})$ in $\RR^{\ell}$. Then, by \eqref{E:V1.V2.orthog},
	\begin{equation} \label{E:rho.V2T.in.omega}
		\rho(V_{2}^{T}) \subset \omega .
	\end{equation}
Recall \eqref{E:A.and.X}. Let $\omega A := \{ wA \in \RR^{k} : w \in \omega \}$.  
By \eqref{E:rank.X.=.k'} and \eqref{E:A.and.X}, we have 
    \begin{equation}  \label{E:Rk.=.rho(X)+rho(A)}
      \RR^{k} = \rho(X) \oplus \rho(A) .
    \end{equation}
Let $\rho(V_{i}^{T} \, A) \subset \rho(A)$ be the, possibly trivial, row space of $V_{i}^{T} \, A$ ($i=1,2$). By \eqref{E:V1T.perp.omega} and the fact that the rows of $A$ are linearly independent, we have
    \begin{equation}  \label{E:rho.A.decomp}
        \rho(A) = \rho( V_{1}^{T} A) \oplus \omega A \subset \RR^{k} .
    \end{equation}
By \eqref{E:A.and.X} and \eqref{E:V1T.perp.omega}, we see that 
$\rho( V_{1}^{T} A)$ and $\omega A$ are orthogonal. \emph{A fortiori,} by \eqref{E:rho.V2T.in.omega}, we have 
    \begin{equation} \label{E:V1T.A.and.V2T.A.orthog}
        \rho( V_{1}^{T} A) \text{ and } \rho(V_{2}^{T} A) \text{ are orthogonal.}
    \end{equation}

Thus, by \eqref{E:Rk.=.rho(X)+rho(A)} and \eqref{E:rho.A.decomp}, we have 
    \begin{equation}  \label{E:breakdown.of.Rk}
      \RR^{k} = \rho(X) \oplus \rho(V_{1}^{T} \, A) \oplus \omega A .
    \end{equation} 
So, if $x_{2} \in \rho(A)$, then we can write $x_{2} = x_{2}' + x_{2}''$ where $x_{2}' \in \rho(V_{1}^{T} \, A)$ 
and $x_{2}'' \in \omega A$. In this case, by \eqref{E:A.and.X}, $X \, (x_{2}')^{T} = 0 = X \, (x_{2}'')^{T}$. 
Then, by \eqref{E:breakdown.of.Rk}, \eqref{E:beta.eps.revealed}, \eqref{E:A.and.X}, 
\eqref{E:rows.of.beta.in.rho.X},  \eqref{E:V1T.A.and.V2T.A.orthog}, and \eqref{E:V1T.perp.omega} we may write
	\begin{multline} \label{E:Phi.Y.eps.x2}
		\Phi(Y_{\epsilon}) 
		  = \bigl\{ ( x_{1} +x_{2}' + x_{2}'', \; x_{1} \hat{\beta} 
	   	   + \epsilon^{-1/2} \,  x_{2}' A^{T} \, V_{1} + x_{2}'' A^{T} \, V_{2}) 
		     \in \RR^{\nvar} :  \\
		     x_{1}^{1 \times k} \in \rho(X), \; x_{2}' \in \rho(V_{1}^{T} \, A), 
			       \text{ and } x_{2}'' \in \omega A \bigr\}.
	\end{multline}
Now make the change of variables $y_{2}' := \epsilon^{-1/2} x_{2}'$. Then \eqref{E:Phi.Y.eps.x2} becomes
	\begin{multline} \label{E:Phi.Y.eps.y2}
		\Phi(Y_{\epsilon}) 
			= \bigl\{ ( x_{1} + \epsilon^{1/2} y_{2} + x_{2}, \, x_{1} \hat{\beta} 
	   	   	  + y_{2} A^{T} \, V_{1} + x_{2} A^{T} \, V_{2} ) \in \RR^{\nvar} : \\
			     x_{1}^{1 \times k} \in \rho(X), \; y_{2} \in \rho(V_{1}^{T} \, A), 
			       \text{ and } x_{2} \in \omega A \bigr\}.
	\end{multline}

Let $\zeta \in G(\ell, \ell+m)$. \emph{Claim:} $V_{1}$, $V_{2}$, and $\omega$ as above can be chosen so that
	\begin{equation}   \label{E:zeta.in.terms.of.omega.V1.V2}
		\zeta = \bigl\{ (w, \, z V_{1} + w V_{2} ) \in \RR^{\ell + m} :    
			      z^{1 \times \ell} \in \rho(V_{1}^{T}) 
			        \text{ ad } w^{1 \times \ell} \in \omega \bigr\}.
	\end{equation}
Let $U^{\ell \times (\ell + m)}$ be a matrix whose row space is $\zeta$. I.e, $\rho(U) = \zeta$. 
Thus, $rank \, U = \ell$. By applying row operations on $U$ if necessary, we may assume that $U$ is an echelon matrix 
(Stoll and Wong \cite[p.\ 46]{rrSetW68.LinearAlgebra}). 

Let $\pi_{1} : \RR^{\ell+m} \to \RR^{\ell}$ be projection onto the first $\ell$ coordinates. Let $L^{\ell \times \ell} := \pi_{1}(U)$ be the matrix obtained from $U$ by applying $\pi_{1}$ row-wise. 
Since $U$ is echelon, the nonzero rows of $L$ are linearly independent. Let $J$ be the set of indices of the nonzero rows of $L$. Thus, $rank \, L$ is the cardinality, $s := |J|$, of $J$. Since $U$ is echelon, we have $J = \{ 1, \ldots, s \}$. I.e., the nonzero rows of $L$ are at the ``top'' (i.e.\ the rows have the lowest row numbers). 

Let $L_{1}^{s \times \ell}$ be the matrix consisting of the nonzero rows of $L$. Hence, $L_{1}$ is of full rank and
	\begin{equation*}
		L =
			\begin{pmatrix}
				L_{1} \\
				0^{(\ell-s) \times \ell}
			\end{pmatrix}.
	\end{equation*}
Let 
	\begin{equation*}
		\omega := \rho(L) = \rho(L_{1}) \subset \RR^{\ell}. 
	\end{equation*}
Then $\dim \omega = s$ and the matrix of orthogonal projection $\RR^{\ell} \to \omega$ 
onto $\omega$ is 
    \begin{equation*}
      Q^{\ell \times \ell} := (R_{L}, 0^{(\ell-s) \times \ell}) L = R_{L} L_{1} ,
    \end{equation*}
where $R_{L}^{\ell \times s} := L_{1}^{T} (L_{1} L_{1}^{T})^{-1}$. 
Thus, $R_{L}$ is of full rank $s$. Set 
	\begin{equation*}
		R^{\ell \times \ell} := ( R_{L}, \; S ),
	\end{equation*}
where $S^{\ell  \times (\ell - s)}$ has rank $\ell - s$ and has column space orthogonal to that of $R_{L}$. Thus, $R$ is of full rank $\ell$. Therefore, $\rho(R U) = \rho(U) = \zeta$. 

Let $\pi_{2} : \RR^{\ell+m} \to \RR^{m}$ be projection onto the last $m$ coordinates and let $V^{\ell \times m}$ be the matrix obtained from $U$ by applying $\pi_{2}$ row-wise. Thus, $U = (L,V)$ and we have 
    	\begin{equation} \label{E:RUQV}
    		R U = (R L, \;RV) = (R_{L} L_{1}, \;RV) = (Q, \;RV).
    	\end{equation}
Replace $U$ by $R U$ -- which in general is not echelon -- so now $L = Q$. This changes $V$ to $RV$, but does not change  
$\zeta = \rho(U)$, or $\omega := \rho(L) = \rho(Q)$. Define $V_{1}^{\ell \times m} := V - QV$ and $V_{2}^{\ell \times m} := QV$ so $V = V_{1} + V_{2}$.
Since $Q$ is an orthogonal projection matrix, we have $Q^{T} = Q$ and $Q^{2} = Q$. \eqref{E:V1.V2.orthog} is immediate.

We prove \eqref{E:V1T.perp.omega}. Now, $V_{1}^{T} L^{T} = (V^{T} - V^{T} Q ) Q = 0$. 
Thus, $\omega \subset \rho(V_{1}^{T})^{\perp}$. Conversely, suppose $x^{1 \times \ell}$ 
is perpendicular to $\rho(V_{1}^{T})$. Then $0 = V_{1}^{T} x^{T} = V^{T} ( x^{T} - Q x^{T} )$ This means
    \begin{equation*}
        U^{T} ( x^{T} - Q x^{T} ) =
            \begin{pmatrix}
                L^{T} \\
                V^{T}
            \end{pmatrix}
        ( x^{T} - Q x^{T} )   =
            \begin{pmatrix}
                Q \\
                V^{T}
            \end{pmatrix}
        ( x^{T} - Q x^{T} )   = 0.
    \end{equation*}
But $U^{\ell \times (\ell + m)}$ has full rank, $\ell$. Therefore, $x^{T} - Q x^{T} = 0$. 
I.e., $x \in \omega$. I.e., \eqref{E:V1T.perp.omega} holds.

We have $\zeta = \{ x U \in \RR^{\ell+m} : x \in \rho(U^{T}) \}$, $L^{T} = Q^{T} = Q = L$. 
In addition, $V_{2}^{T} := V^{T} Q^{T} = V^{T} L^{T}$. 
Thus, $\rho(V_{2}^{T}) \subset \rho(L^{T})$. Therefore, 
    \begin{equation*}
        \rho(U^{T}) =
            \rho \left[
            \begin{pmatrix}
                L^{T} \\
                V^{T}
            \end{pmatrix}
        \right] = \rho \left[
            \begin{pmatrix}
                L ^{T} \\
                V_{1}^{T} + V_{2}^{T}
            \end{pmatrix}
        \right] = \rho \left[
            \begin{pmatrix}
                L \\
                V_{1}^{T}
            \end{pmatrix}
        \right] .
    \end{equation*}
Hence, if $x \in \rho(U^{T})$ we can write uniquely $x = w + z$ with $w \in \rho(L) = \omega$ and $z \in \rho(V_{1}^{T})$. Therefore, by \eqref{E:V1T.perp.omega} and \eqref{E:V1.V2.orthog},  
	\begin{equation*}
		x U = (x L, x V_{1} + x V_{2}) 
		  = \bigl( (w + z) Q, (w + z) V_{1} + (w + z) V_{2} \bigr)
			= (w, z V_{1} + w V_{2} \bigr). 
	\end{equation*}
Thus, \eqref{E:zeta.in.terms.of.omega.V1.V2} holds as claimed. 
	
Let $\xi$ be the plane
	\begin{equation} \label{E:xi.rho.X}
		\xi := \bigl\{ ( x_{1}, \, x_{1} \hat{\beta}) \in \RR^{\nvar} : 
		   x_{1}^{1 \times k} \in \rho(X) \bigr\} \in G(k', \nvar).
	\end{equation}
We have
	\begin{equation*} 
		\xi := \bigl\{ ( yX, yX \hat{\beta}) \in \RR^{\nvar} : 
		   y \in \RR^{n} \text{ a row vector } \bigr\}.
	\end{equation*}

Recall that $k' = rank \, X < k$ and $\ell := k - k'$. Map $\RR^{\ell+m}$ into $\RR^{k +m}$ by 
	\begin{equation*}
		F: (u, v) \mapsto (uA, v), 
		       \quad u^{1 \times \ell} \in \RR^{\ell}, \; v^{1 \times m} \in \RR^{m}.
	\end{equation*}
Recall \eqref{E:A.and.X} and \eqref{E:X.eps.defn}. Since $A^{\ell \times k}$ has orthonormal rows $F$ is an isometric imbedding of $\RR^{\ell+m}$ into $\RR^{k +m} = \RR^{\nvar}$. 
We show that $F(\RR^{k-k'+m}) \cap \xi = \{ 0 \}$. Suppose for some 
$y^{1 \times n}$, $u^{1 \times \ell}$, and $v^{1 \times m}$ we have 
$( yX, yX \hat{\beta}) = (uA, v)$. This means, by \eqref{E:A.and.X}, we have 
$0 = yX A^{T} = u$. Thus, $yX = 0$. Hence, $( yX, yX \hat{\beta}) = 0$, as desired.

By \eqref{E:A.and.X} and \eqref{E:zeta.in.terms.of.omega.V1.V2}, 
	\begin{align*}
	    F(\zeta) &= \Bigl\{ \bigl( w A, \, (z A) A^{T}V_{1} + (w A) A^{T} V_{2} \bigr) 
	              \in \RR^{\nvar} :    \\
		         & \quad \quad \quad \quad \quad \quad z^{1 \times \ell} \in \rho(V_{1}^{T}) 
			            \text{ and } w^{1 \times \ell} \in \omega) \Bigr\} \\
			&= \bigl\{ ( x_{2}, \, y_{2} A^{T} \, V_{1} + x_{2} A^{T} \, V_{2} ) 
			                   \in \RR^{\nvar} :    \\
		          & \quad \quad \quad \quad \quad \quad 
			                         y_{2} \in \rho(V_{1}^{T} \, A), 
			                           \text{ and } x_{2} \in \omega A \bigr\}
	\end{align*}
Thus, from \eqref{E:Phi.Y.eps.y2} and \eqref{E:xi.rho.X} we see that 
	\begin{equation*}
		\Phi(Y_{\epsilon}) \to \xi \oplus F(\zeta) \text{ as } \epsilon \to 0.
	\end{equation*}
Since $Y_{\epsilon} \in \mcl{Y}'$ the proof of the proposition is concluded.
  \end{proof}
 
  	\begin{proof}[Proof of lemma \ref{L:power.diff.lin.indep}]
Let $1 \leq i_{1} < \ldots < i_{\ell} \leq n$. Reordering the $z_{i}$'s if necessary, we may assume $i_{j} = j$. The first entry in every $w_{i} - w_{i_{1}}$ is 0 ($i=2, \ldots, \ell$). Let $w_{1,i}$ be the vector $w_{i}$, but with the first entry removed. Then we must show that $w_{1,2} - w_{1,1}, \ldots, w_{1,\ell} - w_{1,1}$ are linearly independent. Let $W^{(\ell-1) \times (\ell-1)}$ be the matrix whose rows are $w_{1,2} - w_{1,1}, \ldots, w_{1,\ell} - w_{1,1}$. Suppose $w_{1,2} - w_{1,1}, \ldots, w_{1,\ell} - w_{1,1}$ are linearly \emph{de}pendent. Then $W$ has rank less than $\ell-1$. Therefore, for some $a_{1}, \ldots, a_{\ell-1} \in \mathbb{C}$ (the complex numbers), not all 0, 
we have $W a = 0^{(\ell-1) \times 1}$, where 
$a^{(\ell-1) \times 1} := (a_{1}, \ldots, a_{\ell-1})^{T}$. I.e., for each $i = 2, \ldots, \ell$ we have
	\begin{equation*}
		0 = \sum_{j=1}^{\ell-1} a_{j} (z_{i}^{j} - z_{1}^{j})
			= (z_{i} - z_{1}) \sum_{j=1}^{\ell-1} a_{j} 
			  \sum_{m=1}^{j} z_{i}^{j-m} z_{1}^{m-1} 
	\end{equation*}
Since $z_{1}, \ldots, z_{n}$ are distinct, we have 
	\begin{equation*}
		0 = \sum_{j=1}^{\ell-1} a_{j} \sum_{m=1}^{j} z_{i}^{j-m} z_{1}^{m-1},  
				\quad i = 2, \ldots, \ell .
	\end{equation*}
Making the change of variables $h := j-m$, this becomes
	\begin{equation*}
	  0 = \sum_{h=0}^{\ell-2} \left( \sum_{j=h+1}^{\ell-1} a_{j}  
		z_{1}^{j-h-1} \right) z_{i}^{h}, \quad i = 2, \ldots, \ell.
	\end{equation*}
Thus, a polynomial of degree $\ell-2$ has $\ell-1$ distinct roots, $z_{2}, \ldots, z_{\ell}$. (The inner sums do not depend on $z_{i}$ and so can function as coefficients.) Therefore,
	\begin{equation}  \label{E:coefs.must.be.0}
		\sum_{j=h+1}^{\ell-1} a_{j}  z_{1}^{j-h-1} = 0,  \quad h = 0, \ldots, \ell-2.
	\end{equation}
Taking $h = \ell-2$, we find that $a_{\ell-1} = 0$. (Remember that $z_{1} \neq 0$.)  Hence, we may replace $\ell$ by $\ell-1$ in \eqref{E:coefs.must.be.0}, etc. All the $a_{j}$'s are 0. Contradiction.
	\end{proof}

  \begin{proof}[Proof of proposition \ref{P:codim.S.LAD.=1.n-k.even}]
One might be able to prove the proposition using the theory of semi-algebraic sets (Bochnak \emph{et al} \cite{jBmCr-rR98.RealAlgGeom}), but since I know virtually nothing about that subject, I will use elementary methods.

Let $\Ss_{LAD}$ be the singular set of LAD. Let $\Ss_{LAD, C}$ be the set of collinear singularities of LAD. By proposition \ref{P:few.collin.LAD.sings}, and \eqref{E:n>nvar>k>0} $codim \, \Ss_{LAD, C} \geq 3$. Thus, it suffices to show that the set, $\Ss_{LAD, NC}$, of \emph{non}-collinear singularities is 1. 
First, we show that 
    \begin{equation}  \label{E:codim.S.LAD.geq.1}
      codim \, \Ss_{LAD, NC} \geq 1
    \end{equation}

Let $Y = (X^{n \times k}, y^{n \times 1}) \in \Y$ be non-collinear. We will think of $Y$ as a possible element of $\Ss_{LAD, NC}$. As usual, let $x_{i}$ 
be the $i^{th}$ row and $y_{i}$ the $i^{th}$ entry in $X$, $y$, resp.\ 
($i \in \NN_{n}$). Let $J \subset I$ be non-empty. Let $X_{J}^{\nvar \times k}$ be the submatrix of $X$ consisting of rows $x_{i}$ with $i \in J$. Define $y_{J}$ similarly. Let $X_{J 1}^{\nvar \times \nvar} := (1^{\nvar}, X_{J})$. Define 
    \begin{equation}  \label{E:yi-xi.beta.j=0}
      \beta^{J}(Y)^{\nvar \times 1} = \beta^{J} = X_{J 1}^{-1} \; y_{J}
        \text{ so } y_{i} - x_{i} \beta^{j} = 0 \quad (i \in J) .
    \end{equation}
 
Recall, by \eqref{E:N.sub.n}, $\NN_{n} := \{1, \ldots, n \}$. 
Let $| \cdot |$, when applied to sets, denote cardinality. For $j = 1,2$ 
let $J_{j} \subset \NN_{n}$ have cardinality $\nvar$, i.e., $|J_{j}| = \nvar$. 
Suppose $J_{1} \neq J_{2}$. (Possible by \eqref{E:n>nvar>k>0}.) Let $j = 1,2$. By  
lemma \ref{L:basic.LAD.soln.facts} and the comment after it, we may assume at least one of $X_{J_{j} 1}$ ($j=1,2$) has full rank $\nvar$. 

Suppose $X_{J_{1} 1}$ has full rank $\nvar$, but for no choice of $J_{2} \neq J_{1}$ does $X_{J_{2} 1}$ have full rank. By \eqref{E:exact.fit.basis} there exists 
$\beta = \bigl( \beta_{0}, (\beta_{1})^{T} \bigr)^{T} \in \hat{B}(Y)$ s.t.\  for every 
$i \in J_{1}$, 
$y_{i} = \beta_{0} + x_{i} \beta_{1}$. Since $X_{J_{1} 1}$ has full rank there is only one vector 
$\beta$ with this property. We wish that $Y$ possibly be an LAD singularity. Suppose it were one. Then, in accordance with corollary \ref{C:nonunique.mins.are.sings}, 
$\hat{B(Y)}$ would contain another LAD solution 
$\beta' = \bigl( \beta'_{0}, (\beta'_{1})^{T} \bigr)^{T}$ (so $\beta' \neq \beta$). But by corollary \ref{C:B1.is.convex.hull.of.E} and lemma \ref{L:basic.LAD.soln.facts}(iii) there would exits a $J_{2}$ with $|J_{2}| = \nvar$ s.t.\ $X_{J_{2} 1}$ has full rank and for every 
$i \in J_{2}$ we have $y_{i} = \beta'_{0} + x_{i} \beta'_{1}$. But $\beta' \neq \beta$. This implies $J_{2} \neq J_{1}$, contradicting our assumption that $J = J_{1}$ is the only subset of $1, \ldots, n$ of cardinality $\nvar$ for which $X_{J 1}$ has full rank. Since we wish to leave open the possibility that $Y$ is an LAD singularity, we will assume both $X_{J_{1} 1}$ and $X_{J_{2} 1}$ have full rank 
$\nvar$.
 
See \eqref{E:yi-xi.beta.j=0}. Define 
    \begin{equation}  \label{E:beta.Jj}
      \beta^{j}(Y)^{\nvar \times 1} := \beta^{J_{j}} := X_{J_{j} 1}^{-1} \; y_{J_{j}}
        \text{ so } y_{i} - x_{i} \beta^{j} = 0 \quad (i \in J_{j}; j = 1,2) .
    \end{equation}
Write 
    \begin{equation*}
      \beta^{J_{j}} :=  \beta^{j} := 
        \begin{pmatrix}
           (\beta^{j}_{0})^{1 \times 1} \\
           (\beta^{j}_{1})^{k \times 1}
        \end{pmatrix} .
    \end{equation*}

WLOG, $J_{1} = \{ 1, \ldots, \nvar \}$ Again, to make  $Y$ singularity-like, 
assume $\beta^{2}_{1} \neq \beta^{1}_{1}$. ($\beta^{2}_{1} \neq \beta^{1}_{1}$ is a necessary, but not sufficient condition that $Y \in \Ss_{LAD, NC}$.) WLOG 
$\beta^{1}_{1} \neq 0$. 

\emph{Claim:} There exists a $1 \times k$ row vector  $z$ satisfying 
    \begin{equation} \label{E:z.beta.1.and.z.(beta.1-beta.2).not.0}
        z \beta^{1}_{1} \neq 0 \text{ and }
          z (\beta^{1}_{1} - \beta^{1}_{2}) \neq 0.
    \end{equation}
By assumption, $b := \beta^{1}_{1} \neq 0$, $c := \beta^{1}_{1} - \beta^{1}_{2} \neq 0$. Suppose $b$ and $c$ are linearly dependent, e.g. if $k=1$. 
Then for some $\mu =1, \ldots, k$, neither of the $\mu^{th}$ coordinates of $b$ and $c$ are 0. In that case, take $z$ to be the row vector all of whose coordinates are 0 except the $\mu^{th}$ which is 1. Now suppose $b$ and $c$ are linearly independent. 
Then $k \geq 2$ and for some $\mu, \nu =1, \ldots, k$ distinct, the $\mu^{th}$ and $\nu^{th}$ rows of $(b,c)^{k \times 2}$ are linearly independent. 
WLOG $\mu=1$, $\nu = 2$. Let $C^{2 \times 2}$ be the matrix consisting of the first 2 rows of $(b,c)^{k \times 2}$. Let $v^{1 \times 2}$ satisfy $v C = (1,1)$. 
Let $z^{1 \times k} = (v, 0^{1 \times (k-2)})$. This proves the claim.

Let $Z^{n \times k}$ have $z$ as first row with all the other rows 0. Let $X(t) := X + t Z$, 
$t \in \RR$. The first row of $X(t)$ is $x_{1}(t) := x_{1} + t z$. All the other rows of $X(t)$ are the same as those of $X$. Define $X_{J_{j}}(t)^{\nvar \times k}$ in the obvious way:  $X_{J_{j}}(t)$ is the matrix whose rows are all the $i^{th}$ rows of $X(t)$ with $i \in J_{j}$. Recall that $1 \in J_{1}$. Then the first row of $X_{J_{1} 1}(t)$ is 
$\bigl( 1, x_{1}(t) \bigr) := (1, x_{1} + t z)$. All the other rows are the same as those of $X_{J_{1} 1}$. Since $1 \notin J_{2}$ we have $X_{J_{2} 1}(t) \equiv X_{J_{2} 1}$.

Let $w = \bigl( 1, x_{1}(1) \bigr) \beta^{1}(Y) - y_{1}$. 
Thus, since $1 \in J_{1}$, by \eqref{E:beta.Jj}, 
    \begin{equation} \label{E:tw.=.tz.beta.1}
    w = \bigl( 1, x_{1}(1) \bigr) \beta^{1}(Y) - y_{1} 
      = \bigl[ ( 1, x_{1} \bigr) \beta^{1}(Y) - y_{1} \bigr] 
        + (0, z) \beta^{1}(Y) = 0 + z \beta^{1}_{1}(Y) \in \RR .
    \end{equation} 
It follows from \eqref{E:z.beta.1.and.z.(beta.1-beta.2).not.0} that $w \neq 0$. 

Let $y(t)^{n \times 1}$ be the same as $y$, but with $y_{1}$ replaced 
by $y_{1}(t) := y_{1} + t w$, so $y(0) = y$. Let $Y(t) = \bigl( y(t), X(t) \bigr)$. Therefore, $Y(t)$ differs from $Y$ only in the first row. 
Hence, $\beta^{2} \bigl( Y(t) \bigr) = \beta^{2}$. Moreover, by \eqref{E:tw.=.tz.beta.1}, 
    \begin{equation*}
      y_{1}(t) - \bigl( 1, x_{1}(t) \bigr) \beta^{1}   
        = y_{1} + t w - (1, x_{1} + t z)   \beta^{1} 
          = \big[ y_{1} - (1, x_{1}) \beta^{1} \bigr] 
            + \bigl[ t w -  t z \beta^{1} \bigr]  = 0.
    \end{equation*} 
Hence, we also have $\beta^{1} \bigl( Y(t) \bigr) = \beta^{1}$.

Let $\epsilon_{ij} = \pm 1$ ($i  \in \NN_{n}$; $j=1,2$) be arbitrary. Let 
$\blds{\epsilon}^{2 \times n}$ be the matrix $(\epsilon_{ij})$ and for $Y' \in \Y$ define 
    \begin{equation*}
      \Delta(Y'; \blds{\epsilon}) := \Delta_{J_{1}, J_{2}}(Y'; \blds{\epsilon}) 
         = \sum_{i \notin J_{1}} \epsilon_{i1} \bigl( y'_{i} - (1,x'_{i}) \beta^{1}(Y') \bigr)
           - \sum_{i \notin J_{2}} \epsilon_{i2} \bigl( y'_{i} - (1,x'_{i}) \beta^{2}(Y') \bigr)
             \in \RR .
    \end{equation*}
Since $1 \in J_{1} \setminus J_{2}$ and $Y(t)$ differs from $Y$ only in the first row, we have 
    \begin{align*}
     \Delta \bigl( Y(t) ; \blds{\epsilon} \bigr) 
         &= \sum_{i \notin J_{1}} \epsilon_{i1} \bigl( y_{i}(t) - (1,x_{i}(t)) \beta^{1}(Y(t)) \bigr)
           - \sum_{i \notin J_{2}} \epsilon_{i2} \bigl( y_{i}(t) - (1,x_{i}(t)) \beta^{2}(Y(t)) \bigr) \\
         &= \sum_{i \notin J_{1}} \epsilon_{i1} \bigl( y_{i} - (1,x_{i}) \beta^{1} \bigr) 
        - \sum_{i \notin J_{2}, i \neq 1} \epsilon_{i2} \bigl( y_{i} - (1,x_{i}) \beta^{2} \bigr) \\
         & \qquad \qquad - \bigl( y_{1}(t) - (1,x_{1}(t)) \beta^{2} \bigr) .
    \end{align*}
Now, by \eqref{E:tw.=.tz.beta.1}, 
    \begin{multline*}
      y_{1}(t) - (1,x_{1}(t)) \beta^{2} = (y_{1} + t w) - (1,x_{1} + t z) \beta^{2} \\
        = y_{1} - (1,x_{1}) \beta^{2} + t (w - z \beta^{2}) 
          = y_{1} - (1,x_{1}) \beta^{2} + t z (\beta^{1} - \beta^{2}) .
    \end{multline*}
Thus, by \eqref{E:z.beta.1.and.z.(beta.1-beta.2).not.0}, 
$(d/dt) \Delta \bigl( Y(t) ; \blds{\epsilon} \bigr) \restriction_{t=0} = - z (\beta^{1} - \beta^{2}) \neq 0$, by \eqref{E:z.beta.1.and.z.(beta.1-beta.2).not.0}. Therefore, the real-valued function $\Delta( \cdot ; \blds{\epsilon} )$ has full rank, namely 1.

Let $\mcl{G}_{J_{1}, J_{2}}$ be the set of $Y = (X,y) \in \Y$ s.t.\ $X_{J_{j} 1}$ ($j=1,2$) are of full rank and $\beta^{1}_{1}(Y) \neq \beta^{2}_{1}(Y)$. 
Then $\mcl{G}_{J_{1}, J_{2}}$ is clearly open. Let 
$\mcl{Z}_{J_{1}, J_{2}, \blds{\epsilon}} := \bigl\{ Y \in \mcl{G}_{J_{1}, J_{2}} : \Delta( Y ; \blds{\epsilon} ) = 0 \big\}$. Then, by Boothby \cite[Theorem (5.8), p.\ 79]{wmB75} and  \eqref{E:Haus.dim.s-manif.=.s} in appendix \ref{Chptr:Lip.Haus.meas.dim}, 
$codim \; \mcl{Z}_{J_{1}, J_{2}, \blds{\epsilon}} = 1$. Let $\mcl{Z}$ be the union 
of $\mcl{Z}_{J_{1}, J_{2}, \blds{\epsilon}}$ over all choices of $J_{1} \neq J_{2}$ and 
$\blds{\epsilon} \in \{ -1, +1 \}^{\{1, \ldots, 2n\}}$. By \eqref{E:dim.of.whole.=.max.dim.of.parts}, $codim \, \mcl{Z} = 1$. By corollary \ref{C:nonunique.mins.are.sings}, corollary \ref{C:B1.is.convex.hull.of.E}, and lemma \ref{L:basic.LAD.soln.facts}, the set, $\Ss_{LAD, NC}$, of non-collinear singularities of LAD is a subset of $\mcl{Z}$. Therefore, $codim \, \Ss_{LAD, NC} \geq 1$. 

Before we show $codim \, \Ss_{LAD, NC} \leq 1$ when $n-k$ is even, we prove a couple of facts true whether or not $n-k$ is even. First we prove a local boundedness property. Let also $Y = (X,y) \in \Y$. Let $Y' = (X', y') \in \Y$. Let $\delta > 0$ and suppose 
$\| Y' - Y \| < \delta$. (See \eqref{E:matrix.norm}.)
Let $J \subset \NN_{n} = \{1, \ldots, n \}$ contain $\nvar = k+1$ elements. As above, 
let $X_{J}^{\nvar \times k}$ be the submatrix of $X$ consisting of the rows $x_{i}$ of $X$ indexed by $i \in J$. Define $X'_{J}$, $y_{J}$, and $y'_{J}$ similarly. 
Suppose $X_{J 1}^{\nvar \times \nvar} := (1^{\nvar}, X_{J})$ has full rank $\nvar$. Assume $Y'$ is close enough to $Y$ that the similarly defined $X'_{J 1}$ also has full rank (lemma \ref{L:rank.lwr.semicont}). Suppose $\beta^{\nvar \times 1}$ 
and $(\beta')^{\nvar \times 1} = \beta'(Y')$ satisfy $y_{i} = (1,x_{i}) \beta$ and $y'_{i} = (1,x'_{i}) \beta'$ for $i \in J$. Write $\beta = (\beta_{0}, \beta_{1}^{T})^{T}$, etc., as usual. We do not require $\beta$ or $\beta'$ to be LAD solutions for $Y$ and $Y'$, resp. By lemma \ref{L:basic.LAD.soln.facts}(a),  
        \begin{equation}   \label{E:beta.bounded}
          \text{There exists } \delta > 0 \text{ and } C < \infty \text{ s.t.\ }
             |\beta'(Y')| < C \text{ if } \| Y' - Y \| < \delta.
        \end{equation}
        
Next, we show that 
    \begin{equation}   \label{E:beta'-beta=O(delta)}
      \beta' - \beta = O(\delta) .
    \end{equation}
By \eqref{E:beta.bounded}, 
$(X'_{J 1} - X_{J 1}) (\beta'  - \beta) = O(\delta)$ as $\delta \downarrow 0$. 

We have
    \begin{align*}
       0 &= y'_{J} - X'_{J 1} \beta' \\
          &= y_{J} + (y'_{J} - y_{J}) 
            - \bigl[ X_{J 1} + (X'_{J 1} - X_{J 1})  \bigr] 
              \bigl[ \beta + (\beta'  - \beta  \bigr] \\
          &= (y_{J} - X_{J 1} \beta ) 
            - X_{J 1} (\beta'  - \beta ) - (X'_{J 1} - X_{J 1}) \beta 
              - (X'_{J 1} - X_{J 1}) (\beta'  - \beta) 
                + (y'_{J} - y_{J}) \\
          &= 0 - X_{J 1} (\beta'  - \beta ) + O(\delta) .
    \end{align*}
Thus, 
    \begin{equation*}
      \beta'  - \beta = (X_{J 1})^{-1} O(\delta) .
    \end{equation*}
Since $X_{J 1}$ is of full rank and does not depend on $\delta$, that proves \eqref{E:beta'-beta=O(delta)}.

\emph{Now suppose $n-k$ is even,} say, $n-k = 2 \ell$. Let $e_{1}, \ldots, e_{k}$ be the standard basis of $\RR^{k}$. Identify $\RR^{k-1}$ with the span of $e_{1}, \ldots, e_{k-1}$. So if $k=1$ then $\RR^{k-1} = \{0\}$. Let $Y = (X,y) \in \Y$ with 
        \begin{multline} \label{E:x.i.y.i.choices} 
          x_{1} = 0 \text{ and } x_{2}, \ldots, x_{k} 
            \text{ are linearly independent vectors in } \RR^{k-1} , \\ 
              x_{k+1} = \cdots = x_{n} = e_{k} = e_{\nvar-1} , \\
               \text{ and } y_{1} = \cdots = y_{k} = 0 \text{ and } 
                 y_{k+1}, \ldots, y_{n} \in \RR \text{ are distinct.}  
        \end{multline}
As usual, write $x_{i} = (x_{i1}, \ldots, x_{ik})$. Thus, $x_{i}$ is $1 \times k$ 
($i=1, \ldots, k$) but $x_{1k}= \cdots = x_{kk} = 0$. 
Let $b = \bigl( b_{0}^{1\times 1}, (b_{1}^{k \times 1})^{T} \bigr)^{T}$ and let 
    \begin{equation}   \label{E:prime.double.prime}
      (b_{1}')^{\nvar \times 1} = (b_{11}, \ldots, b_{1(k-1)}, 0)^{T} \text{ and }
        (b_{1}'')^{\nvar \times 1} = (0^{1 \times (k-1)}, b_{1k})^{T} .
    \end{equation}
So $b_{1} = b_{1}' + b_{1}''$ and $b_{1}'$ and $b_{1}''$ are orthogonal. Recall the definition \eqref{E:L1.criterion.defn} of $L^{1}$. We have 
	\begin{equation*}   
		L^{1}(b,Y) = \sum_{i=1}^{n} | y_{i} - b_{0} - x_{i} b_{1} | 
		  =  | b_{0} | + \sum_{i=2}^{k} | b_{0} + x_{i} b_{1}' |
		    + \sum_{i=k+1}^{n} | y_{i} - b_{0} - e_{k} b_{1}'' | .
	\end{equation*}

A choice of $b$ that minimizes each term on the RHS of the preceding separately will be an LAD solution for $Y$. It is possible to do this: Take $b_{0} = 0$, $b_{1}' = 0$, and 
choose $e_{k} b_{1}'' = b_{1k}$ to be a median of $y_{k+1}, \ldots, y_{n}$. Actually, since $x_{1}, \ldots, x_{k}$ are linearly independent vectors in $\RR^{k-1}$, 
0 is the only choice for $b_{1}'$. Thus, by \eqref{E:prime.double.prime}, any LAD fit has the form $b^{\nvar \times 1} = \bigl(0^{1 \times 1}, 0^{1 \times (k-1)} , m \bigr)^{T}$, where $m$ is any median of $y_{k+1}, \ldots, y_{n}$. Since $y_{k+1}, \ldots, y_{n} \in \RR$ are an even number of distinct values, every point in the closed ``middle interval'', when $y_{k+1}, \ldots, y_{n}$ are arranged in increasing order, is a median. Re-index if necessary so that the two middle values are $y_{k+1}$ and $y_{k+2}$. Then 
    \begin{equation}  \label{E:beta.j's.for.Y}
      \beta^{j} = \bigl(0^{1 \times 1}, 0^{1 \times (k-1)} , y_{k+j} \bigr)^{T} , 
        \; (j=1,2) 
    \end{equation} 
are distinct LAD solutions for $Y$. Therefore, by corollary \ref{C:nonunique.mins.are.sings}, $Y$ is a singularity of LAD. 

Let $\delta > 0$ be small and displace each entry in the matrix $Y^{n \times \nvar}$ by an amount no greater than $\delta$ in absolute value. Call the perturbed data set $Y' = (y', X')$. Then $Y' - Y = O(\delta)$. For small enough $\delta$, $Y'$ is non-collinear and we will assume that. Thus, $x_{ik} = O(\delta)$ for $i = 1, \ldots, k$ and $e_{k} - x_{i} = O(\delta)$ for $i = k+1, \ldots, n$. We always take $\delta$ so small that $y'_{k+1}, \ldots, y'_{n}$ are distinct and have the same order as $y_{k+1}, \ldots, y_{n}$.

Let $(\beta')^{\nvar \times 1} = \bigl( \beta'_{0}, (\beta'_{1})^{T} \bigr)^{T}$ be one of the extreme points (perhaps the only point) of the solution set $\hat{B}_{1}(Y')$ as in lemma \ref{L:basic.LAD.soln.facts} and corollary \ref{C:B1.is.convex.hull.of.E}. (The usage of prime ``$'$'' specified by \eqref{E:prime.double.prime} does not apply here. Call the LAD plane $\bigl\{ (x, (1,x) \beta' : x \in \RR^{k} \bigr\}$ corresponding to $\beta'$ an ``extreme LAD plane''.) By lemma \ref{L:basic.LAD.soln.facts}(a), for $\delta > 0$ sufficiently small, $\beta'$ is bounded uniformly in $Y' \in \Y$ with $\|Y' - Y\| < \delta$. 

Since $\beta'$ is extreme, by lemma \ref{L:basic.LAD.soln.facts}(b,(iii)), the LAD plane corresponding to it -- i.e., the graph of $x \mapsto (1, x) \beta'$ -- must pass through at least $\nvar$ data points $(x'_{i}, y'_{i})$. First let us suppose that the LAD plane corresponding to $\beta'$ does \emph{not} pass through $(x'_{1}, y'_{1}), \ldots, (x'_{k}, y'_{k}), (x'_{k+1}, y'_{k+1})$ or through
$(x'_{1}, y'_{1}), \ldots, (x'_{k}, y'_{k}), (x'_{k+2}, y'_{k+2})$. 

Let us refine our supposition and begin by supposing the graph of $x \mapsto (1, x) \beta'$ \emph{does} pass through the $\nvar$ data points (rows of $Y'$) 
indexed by $1, \ldots, k, r$, for some $r > k+2$. Now, $y'_{i}$ is within $\delta$ of $y_{i}$ 
$(i \in \NN_{n}$) and, by \eqref{E:x.i.y.i.choices}, $y_{k+1}, \ldots, y_{n}$ are distinct. Hence, making $\delta > 0$ smaller, if necessary, the differences 
$y'_{i} - y'_{j}$ ($k+1 \leq i < j \leq n$ are bounded away from 0 (as $\delta \downarrow 0$), and hence are also distinct. We also have $| x'_{i} - x'_{r}| \leq |x'_{i} - e_{k}| + |x'_{i} - e_{k}| = | x'_{i} - x_{i}| + | x'_{r} - x_{r}| = O(\delta)$. Moreover, $\beta'$ is bounded. Thus,
    \begin{align*}
      y'_{i} - y'_{r} &= \bigl[ y'_{i} - (1,x'_{i}) \beta' \bigr] -  \bigl[ y'_{r} - (1,x'_{i}) \beta' \bigr] \\
       &= \bigl[ y'_{i} - (1,x'_{i}) \beta' \bigr] -  \bigl[ y'_{r} - (1,x'_{r}) \beta' \bigr] 
         - \bigl[ (1,x'_{r}) - (1,x'_{i}) \bigl] \beta' \\
       &= \bigl[ y'_{i} - (1,x'_{i}) \beta' \bigr] - 0 + O(\delta), \qquad i = k+1, \ldots, n .
    \end{align*}
\emph{Note that this argument is valid if we only require }$r > k$. Hence, for $i = k+1, \ldots, n$, with $i \neq r$, we have $y'_{i} - y'_{r}$ 
and $y'_{i} - (1,x'_{i}) \beta'$ have the same sign. From this we draw two conclusions. First, since $y'_{i} - y'_{r}$ ($k+1 \leq j \leq n$ is bounded away from 0, it follows that
    \begin{equation} \label{E:resid.neq.0.if.i.neq.s}
       \text{If } k < i \neq s \text{ then } y'_{i} - (1,x'_{i}) \beta' \neq 0, 
         \text{ where } s > k, \;  y_{i} = (1, x'_{i}) \beta' 
           \text{ for } i > k, i \neq s .
    \end{equation}
Our second is that, for $\delta > 0$ sufficiently small, $y'_{r}$ is the smallest of $y'_{k+1}, \ldots, y'_{n}$ satisfying 
$y'_{i} \geq (1, x'_{i}) \beta'$. 

Let $z_{1} < \cdots < z_{n-k}$  be the values $y_{k+1}, \ldots, y_{n}$ in increasing order. Recall $n-k = 2 \ell$ and $y_{k+1}$ and $y_{k+2}$ are the two middle values 
in $y_{k+1}, \ldots, y_{n}$. Thus, $z_{\ell} = y_{k+1}$ and $z_{\ell+1} = y_{k+2}$. Assume that $y_{r} = z_{c}$. Thus, there are $c$ $y_{j}$'s no bigger then $y_{r}$ and 
$2 \ell - c - 1$ $y_{j}$'s bigger then $y_{r}$. Thus, $c$ is the rank -- but not necessarily the position -- of $y_{r}$ in the list $y_{k+1}, \ldots, y_{n}$.  

We will examine what happens as we disturb $\beta'$ in a specific way. For $t \in \RR$ let 
    \begin{equation*}
      \beta'(t) = \beta' - t e_{\nvar} = \beta' - t e_{k+1} .
    \end{equation*}
Thus, $\beta'(t) = (\beta'_{0}, \beta'_{1}, \ldots, \beta'_{k-1}, \beta'_{k} - t)^{T}$. 
By assumption, $y'_{i} - (1, x'_{i}) \beta' = 0$ for $i = 1, \ldots, k, r$.  
By \eqref{E:x.i.y.i.choices}, we have
        \begin{equation} \label{E:y'i-(1,x'i).beta'(t),when.i.leq.k}
         y'_{i} - (1, x'_{i}) \beta'(t) = y'_{i} - (1, x'_{i}) \beta' + t x'_{ik}  
           = t x'_{ik} = O(\delta) t \text{ for } i = 1, \ldots, k  .
        \end{equation}
Now let $i > k$ with $i \neq r$. Then, as observed above, we have 
$y'_{i} - (1, x'_{i}) \beta' \neq 0$. Therefore, for $|t|$ sufficiently small, 
$y'_{i} - (1, x'_{i}) \beta'(t) \neq 0$. Moreover, by \eqref{E:x.i.y.i.choices}, 
    \begin{equation}  \label{E:y'i-(1,x'i).beta'(t),when.i.>.k}
      y'_{i} - (1, x'_{i}) \beta'(t) = y'_{i} - (1, x'_{i}) \beta' +  t \bigl( 1 + (x_{i k}-1) \bigr) 
        = y'_{i} - (1, x'_{i}) \beta' +  t + O(\delta) t .
    \end{equation}
Thus, if $i > k$ and $i \neq r$, the quantity $y'_{i} - (1, x'_{i}) \beta'(t)$ has the same sign 
as $y'_{i} - (1, x'_{i}) \beta'$. The same is true 
for $i = r$, providing we assign $y'_{r} - (1, x'_{r}) \beta = 0$ the same sign that $t$ has.

Recall that $y'_{r}$ is the smallest of $y'_{k+1}, \ldots, y'_{n}$ satisfying 
$y'_{i} \geq (1, x'_{i}) \beta'$. Then, by \eqref{E:y'i-(1,x'i).beta'(t),when.i.leq.k} and \eqref{E:y'i-(1,x'i).beta'(t),when.i.>.k}, for small $t \geq 0$, 
    \begin{align}   \label{E:L1.yr.even.case}
      L^{1} \bigl( \beta'(t), Y' \bigr) 
           &= \sum_{i=1}^{k} \bigl| y'_{i} - (1, x'_{i}) \beta'(t) \bigr|
             + \sum_{i > k, y'_{i} \geq y'_{r}} \big( y'_{i} - (1, x'_{i}) \beta'(t) \bigr) \notag \\
               & \qquad \qquad - \sum_{i > k, y'_{i} < y'_{r}} \bigl( y'_{i} - (1, x'_{i}) \beta'(t) \bigr) 
                 \notag \\
          &= |t| \sum_{i=1}^{k} |x'_{ik}| 
            + \sum_{i > k, y'_{i} \geq y'_{r}} \big( y'_{i} - (1, x'_{i}) \beta' \bigr) 
              + (n - k - c + 1)t + O(\delta) t \\
           & \qquad \qquad - \sum_{i > k, y'_{i} < y'_{r}} \big( y'_{i} - (1, x'_{i}) \beta' \bigr) 
               - (c - 1) t + O(\delta) t \notag \\
          &= \left( \sum_{i > k, y'_{i} \geq y'_{r}} \big( y'_{i} - (1, x'_{i}) \beta'(t) \bigr)
               - \sum_{i > k, y'_{i} < y'_{r}} \big( y'_{i} - (1, x'_{i}) \beta'(t) \bigr) \right) \notag \\
          & \qquad \qquad - 2 \bigl( c - \ell - 1 +  O(\delta) \bigr) t  . \notag
    \end{align} 
Suppose $c \geq \ell + 2$. Then, for $\delta$ sufficiently small, 
$2 \bigl( c - \ell - 1 +  O(\delta) \bigr) > 1$. Therefore, $L^{1}(\beta'(t), Y')$ decreases with increasing $t$. But by assumption, 
$L^{1}(\beta'(0), Y') = L^{1}(\beta', Y') \leq L^{1}(\beta'(t), Y')$. But this means  $L^{1}(\beta', Y)$ can be reduced by increasing $\beta'_{1k}$ a small amount. Thus, the assumption $c > \ell + 1$ has led to a contradiction. Similarly, $c < \ell$ also leads to a contradiction. Therefore, $c = \ell$ or $\ell+1$. I.e., if an LAD plane corresponding to an extreme LAD solution  passes through $(x_{1}, y_{1}), \ldots, (x_{k}, y_{k})$ it must also pass through $(x_{k+1}, y_{k+1})$ or $(x_{k+2}, y_{k+2})$. (If that plane is non-extreme it must pass between those points.)

Now suppose the LAD plane corresponding to the extreme LAD solution $\beta'$ does not pass through all the points 
$(x'_{1}, y'_{1}), \ldots, (x'_{k}, y'_{k})$. Then by lemma \ref{L:basic.LAD.soln.facts}(b,(iii)) it must pass through at least \emph{two} of $(x'_{i}, y'_{i})$ ($i = k+1, \ldots, n$), say, through $(x'_{\mu}, y'_{\mu})$ and $(x'_{\nu}, y'_{\nu})$, where $\mu, \nu = k+1, \ldots, n$ are distinct. \emph{And} we may assume $x'_{\nu} \neq x'_{\mu}$. 

Therefore, $x'_{\nu} - x'_{\mu} \neq 0$. 
However, it is true that $x'_{\nu} - x'_{\mu} = O(\delta) \neq 0$ because, by \eqref{E:x.i.y.i.choices}, $x'_{\nu}$ and $x'_{\mu}$ are both within $\delta$ of $e_{k}$. 
On the other hand, $y'_{\mu} - y'_{\nu} = y_{\mu} - y_{\nu} + O(\delta)$ and $y_{\mu} - y_{\nu} \neq 0$ and does not depend on $\delta$. Thus, 
    \begin{equation*}
      O(\delta) \beta' = \bigl[ (1,x'_{\nu}) - (1,x'_{\mu}) \bigr] \beta' 
        = y'_{\nu} - y'_{\mu} = O(1).
    \end{equation*} 
I.e., $\beta' = O(1/\delta)$, in a nontrivial way so that $\beta' \to \infty$ as $\delta \to 0$. 
This contradicts lemma \ref{L:basic.LAD.soln.facts}(a) and \eqref{E:beta.bounded}.\footnote{That argument only works if as $\delta \to 0$, the LAD plane continues to pass 
through $(x'_{\mu}, y'_{\mu})$ and $(x'_{\nu}, y'_{\nu})$, which may not be the case. (But we insist that the LAD plane does never passes through all the points 
$(x'_{1}, y'_{1}), \ldots, (x'_{k}, y'_{k})$.) But there are only $n < \infty$ data points, where $n$ does not depend on $\delta$. Hence, there exists a sequence $\delta_{m} \to 0$ and $j, j' \in \{ k+1, \ldots, n \}$ distinct such that for every $m$ the plane \emph{does} pass through $(x'_{j}, y'_{j)})$ and $(x'_{j'}, y'_{j'})$. WLOG $j = \mu$ and $j' = \nu$.} 
We conclude that for $\delta$ sufficiently small, no LAD plane can pass through any two of the points $(x'_{k+1}, y'_{k+1}), \ldots, (x'_{n}, y'_{n})$. Hence, any
 LAD plane must pass through all of $(x_{1}, y_{1}), \ldots, (x_{k}, y_{k})$.
 
To sum up, 
    \begin{multline} \label{E:passes.through.middle.2}
      \text{ For small $\delta > 0$, any extreme LAD plane of } Y' 
      \text{ passes through one of } \\
        (x'_{1}, y'_{1}), \ldots, (x'_{k}, y'_{k}), (x'_{k+j}, y'_{k+j}), j = 1 \text{ or } 2,
    \end{multline}
where $y'_{k+1}$ and $y'_{k+2}$ are the middle two values in $y'_{k+1}, \ldots, y'_{n}$. 

As before, let $Y' = (y', X') \in \Y'$ satisfy $Y' - Y = O(\delta)$. Let $J_{1}, J_{2} \subset \NN_{n}$ be distinct subsets of cardinality $\nvar = k +1$. Let $j = 1,2$. Suppose $b^{\nvar \times 1} = \beta^{\prime j}$ satisfies $y_{i} - (1,x_{i}) b = 0$ ($i \in J_{j}$). If $J_{j}$ is neither $\{ 1, \ldots, k, h \}$, where $h = k+1, k+2$, 
then $\beta^{\prime j}$ is not an LAD solution for $Y'$. So let
    \begin{equation}   \label{E:Jj.defn}
      J_{j} = \{1, \ldots, k, k+j \}, \qquad j = 1,2 . 
    \end{equation} 
Thus, the graph of $x \mapsto (1,x) \beta^{\prime j}$ passes through $(x'_{1}, y'_{1}), \ldots, (x'_{k}, y'_{k}), (x'_{k+j}, y'_{k+j})$ ($j=1,2$). From what we have just seen, one or both of $\beta^{\prime 1}$ and $\beta^{\prime 2}$ is an LAD solution for $Y'$.

By \eqref{E:beta.j's.for.Y}, $|\beta^{2} - \beta^{1}| =  y_{k+2} - y_{k+1} > 0$. Therefore, by \eqref{E:beta'-beta=O(delta)}, for small $\delta > 0$, 
    \begin{equation}  \label{E:beta'1.beta'2.bounded.apart}
      |\beta^{\prime 2} - \beta^{\prime 1}| >  (y_{k+2} - y_{k+1})/2 > 0 .
    \end{equation} 

For $t \in \RR$ let $\tilde{Y}_{t} = \tilde{Y}(t)$ be the same as $Y'$ except with $(x'_{k+1}, y'_{k+1})$ replaced by $(\tilde{x}_{k+1}, \tilde{y}_{k+1})$, where
    \begin{equation}   \label{E:x.tilde.and.y.tilde}
      \tilde{x}_{k+1} = x'_{k+1} + t e_{k} \text{ and } 
        \tilde{y}_{k+1} := y'_{k+1} + t (0, e_{k}) \beta^{\prime 1} .
    \end{equation}
    
Since all rows of $\tilde{Y}_{t}$ are the same as those of $Y'$ except the $(k+1)^{st}$, for $j=2$, the graph of $x \mapsto (1,x) \beta^{\prime j}$ still passes 
$(\tilde{x}_{1}, \tilde{y}_{1}), \ldots, (\tilde{x}_{k}, \tilde{y}_{k}), (\tilde{x}_{k+j}, 
\tilde{y}_{k+j})$. In fact, the same is true for $j=1$: 
    \begin{multline} \label{E:tilde.y.k+1-(0,tilde.x.k+1).beta'.1=0}
      \tilde{y}_{k+1} - (1, \tilde{x}_{k+1}) \beta^{\prime 1}
        = y'_{k+1} + t (0, e_{k}) \beta^{\prime 1} 
          - (1, x'_{k+1} + t e_{k}) \beta^{\prime 1} \\
            = \bigl( y'_{k+1} - (1, x'_{k+1}) \beta^{\prime 1} \bigr) 
              +  t (0, e_{k}) \beta^{\prime 1} - t (0, e_{k}) \beta^{\prime 1} = 0 .
    \end{multline}
Thus, $\beta^{\prime 1}$ and $\beta^{\prime 2}$ are possible extreme LAD solutions 
for $\tilde{Y}$ and as we have seen one or both are such solutions, at least if $|t|$ is small. 

For any $t \in \RR$,
        \begin{multline*}
          \tilde{y}_{k+1} - (1, \tilde{x}_{k+1}) \beta^{\prime 2} 
            = y'_{k+1} + t (0, e_{k}) \beta^{\prime 1}
              - (1, x'_{k+1}) \beta^{\prime 2} - t(0, e_{k}) \beta^{\prime 2} \\             
                = \bigl( y'_{k+1} - (1, x'_{k+1}) \beta^{\prime 2} \bigr)
                  + t(0, e_{k}) (\beta^{\prime 1} - \beta^{\prime 2}) . 
        \end{multline*} 
Thus, by \eqref{E:beta.bounded}, $t (0, e_{k}) (\beta^{\prime 1} - \beta^{\prime 2}) \to 0$ 
as $t \to 0$, providing $\delta > 0$ is small.

By \eqref{E:beta.j's.for.Y}, \eqref{E:beta'-beta=O(delta)}, and \eqref{E:x.i.y.i.choices}, 
    \begin{equation*}
      y'_{k+1} - (1, x'_{k+1}) \beta^{\prime 2} 
        = y_{k+1} - (1, x_{k+1}) \beta^{2} + O(\delta)
          = y_{k+1} - y_{k+2} + O(\delta).
    \end{equation*}
We may assume $y_{k+2} > y_{k+1}$. Hence, $y'_{k+1} - (1, x'_{k+1}) \beta^{\prime 2}$ is also $< 0$ for small $\delta$. In conclusion, for $\delta > 0$ and $|t|$ sufficiently small, 
    \begin{multline} \label{E:2.parts.of.y.tilde.k+1-(1,x.tilde.k+1)beta2}
      \bigl| \tilde{y}_{k+1} - (1, \tilde{x}_{k+1}) \beta^{\prime 2} \bigr|
        = - \bigl( y'_{k+1} - (1, x'_{k+1}) \beta^{\prime 2} \bigr) 
            - t(0, e_{k}) (\beta^{\prime 1} - \beta^{\prime 2}) \\
          = - \bigl( y'_{k+1} - (1, x'_{k+1}) \beta^{\prime 2} \bigr) 
            + t(0, e_{k}) (\beta^{\prime 2} - \beta^{\prime 1}) .
    \end{multline}
 
We know that $\bigl| y'_{i} - (1, x'_{i}) \beta^{\prime j} \bigr| = 0$ for $i = 1, \ldots, k, k+j$ ($j=1,2$). Hence, by \eqref{E:2.parts.of.y.tilde.k+1-(1,x.tilde.k+1)beta2} and 
\eqref{E:tilde.y.k+1-(0,tilde.x.k+1).beta'.1=0}, for $\delta > 0$ and $|t|$ small we have
    \begin{align}  \label{E:Delta(t)=t.beta.diff+Delta(0)}
       L^{1}(\tilde{\beta}^{2}, &\tilde{Y}_{t}) - L^{1}(\tilde{\beta}^{1}, \tilde{Y}_{t}) \notag \\
        &= L^{1}(\beta^{\prime 2}, \tilde{Y}_{t}) - L^{1}(\beta^{\prime 1}, \tilde{Y}_{t})  \\
        &= t(0, e_{k}) (\beta^{\prime 2} - \beta^{\prime 1})
            - \bigl( y'_{k+1} - (1, x'_{k+1}) \beta^{\prime 2} \bigr) \notag \\
        &\qquad \qquad + \sum_{i = k+2}^{n} \bigl| y'_{i} - (1, x'_{i}) \beta^{\prime 2} \bigr|
            - \sum_{i = k+1}^{n} \bigl| y'_{i} - (1, x'_{i}) \beta^{\prime 1} \bigr| \notag .
    \end{align}
Here we use the fact that, by \eqref{E:tilde.y.k+1-(0,tilde.x.k+1).beta'.1=0}, 
    \begin{equation} 
      y'_{k+1} - (1, x'_{k+1}) \beta^{\prime 1} = 0 
        = \tilde{y}_{k+1} - (1, \tilde{x}_{k+1}) \beta^{\prime 1} . 
    \end{equation} 
Let $\epsilon_{ij}(Y') := \epsilon_{ij} := sign \bigl( y'_{i} - (1, x'_{i}) \beta^{\prime j} \bigr)$. (See \eqref{E:sign.function}.) By \eqref{E:resid.neq.0.if.i.neq.s}, for $\delta$ sufficiently small and $i \in J_{j}^{c} := \NN_{n} \setminus J_{j}$ we have 
$y'_{i} - (1, x'_{i}) \beta^{\prime j} \neq 0$. Hence, for $\delta$ and $|t|$ sufficiently small and $i \in J_{j}^{c} := \NN_{n} \setminus J_{j}$ we have $\epsilon_{ij}(Y')$ is constant in $Y'$ and $t$. Let $\Delta(Y') := L^{1}(\beta^{\prime 2}, Y') - L^{1}(\beta^{\prime 1}, Y')$. Then
    \begin{equation*}
            \Delta(Y') 
                := \sum_{i \in J_{2}^{c}} \epsilon_{i2} 
                   \bigl( y'_{i} - (1, x'_{i}) \beta^{\prime 2} \bigr)
                     - \sum_{i \in J_{1}^{c}} \epsilon_{i1} 
                       \bigl( y'_{i} - (1, x'_{i}) \beta^{\prime 1} \bigr) \in \RR.
    \end{equation*}
(We have $\Delta(Y) = 0$ because in the case $Y' = Y$, we have that $\beta^{\prime 2}$ and $\beta^{\prime 1}$ are LAD solutions.) Note that $\Delta(Y')$ is differentiable in $Y'$. Let $v(Y')^{n \times \nvar}$ be the matrix all of whose entries are 0 except the entry in the $(k+1)^{st}$ row and $k^{th}$ column is 1 and the entry in the $(k+1)^{st}$ row and $\nvar^{th}$ column is $(0, e_{k}) \beta^{\prime 1}$. From \eqref{E:x.tilde.and.y.tilde}, \eqref{E:Delta(t)=t.beta.diff+Delta(0)}, and \eqref{E:beta'1.beta'2.bounded.apart} we know that the derivative of $\Delta(Y') $ in the direction $v(Y')$ is nonzero. Therefore, by Boothby \cite[Theorem (5.8), p.\ 79]{wmB75}, the set of $Y'$ in a small neighborhood of $Y$ for which $\Delta(Y') = 0$ is a $(n \nvar -1)$-dimensional smooth manifold. Call that manifold $\Ss_{Y,\delta}$. 
By \ref{E:Haus.dim.s-manif.=.s}, $codim \, \Ss_{Y,\delta} = 1$. ($codim$ is based on Hausdorff dimension.)

But by \eqref{E:beta'1.beta'2.bounded.apart} again, for $\delta$ sufficiently small, $\beta^{\prime 1} \neq \beta^{\prime 2}$. Therefore, by \eqref{E:passes.through.middle.2} and corollary \ref{C:nonunique.mins.are.sings}, we have 
$\Ss_{Y,\delta} \subset \Ss_{LAD}$. 
This proves that in the $n - k$ even case, $\Ss_{LAD}$ has co-dimension no greater than 1, hence exactly 1, by \eqref{E:codim.S.LAD.geq.1}. 
  \end{proof}

  \begin{proof}[Proof of lemma \ref{L:Bd.va.o(A,v0).does.not.lie}] For suppose 
    \begin{multline} \label{E:Br(p).contains.Bd.o(A)}
      \text{There exists } p \in P \text{ and } r \in [0, s) \\
        \text{ s.t.\ } Bd_{v_{\msf{A}}} \, o(\msf{A},v_{0}) \text{ lies in the ball } 
          B_{r}(p) \subset P \text{ centered at } p \text{ with radius } r < s.
    \end{multline}
This means $p \neq c$. We show that this leads to a contradiction to \eqref{E:theta.v0.defn}. 

Suppose $p = 0$. Since $p \in P$ that means the plane $P$ passes through the origin, which means $\cos \theta_{\msf{A}} = 0$ (so $\theta_{\msf{A}} = \pi/2$), which means $c = 0 = p$, which means $r = 1 = \sin \theta_{\msf{A}} = s$, 
contradicting $r < s$. Therefore, 
    \begin{equation*}
      p \neq 0 . 
    \end{equation*}

We show that $p$ may be chosen arbitrarily close to $c$. 
Let $\lambda \in (0,1)$ so  $(1-\lambda) c + \lambda p \in P$, because $P$ is convex. 
If $x \in Bd_{v_{\msf{A}}} \, o(\msf{A},v_{0})$ then $|x-c| =s$ and we have
    \begin{align*}
      \Bigl| x - \bigl[ (1-\lambda) c + \lambda p \bigr] \Bigr|^{2}
        &= \bigl| (x-c) + \lambda (c-p) \bigr|^{2} \\
        &= |x-c|^{2} + 2 \lambda (x-c) \cdot (c-p) + \lambda^{2} |c-p|^{2} \\
        &= |x-c|^{2} + \lambda \bigl[ 2 (x-c) \cdot (c-p) + \lambda |c-p|^{2} \bigr] \\
        &\leq |x-c|^{2} + \lambda \bigl[ 2 (x-c) \cdot (c-p) + |c-p|^{2} \bigr] \\
        &= (1-\lambda) |x-c|^{2} 
          + \lambda \bigl[ |x-c|^{2} + 2 (x-c) \cdot (c-p) + |c-p|^{2} \bigr] \\
        &= (1-\lambda) |x-c|^{2} + \lambda \bigl| (x-c) + (c-p) \bigr|^{2} \\
        &= (1-\lambda) |x-c|^{2} + \lambda |x-p|^{2} \\
        &\leq (1-\lambda) s^{2} + \lambda r^{2} < s^{2} .
   \end{align*}
Thus, we may replace $p$ by $(1-\lambda) c + \lambda p \in P$ 
and it remains the case that $Bd_{v_{\msf{A}}} \, o(\msf{A},v_{0})$ lies in a ball centered at $p$ with radius $< s$. (Moving $p$ close to $v_{\msf{A}}$ may require increasing $r$, but still $r < s$.) 
    
Next, we show that if $Bd_{v_{\msf{A}}} \, o(\msf{A},v_{0}) \subset P$ lies in a ball centered at $p$ with radius $r < s$ then it lies in an open hemisphere 
of $S \subset P$ (in fact in a closed spherical cap inside an open hemisphere). (Recall that by \eqref{E:Bd.o(A).S.sin.theta.A}, $S = P \cap S^{k}$ is the $(k-1)$-sphere centered at $c$ with radius $s := \sin \theta_{\msf{A}}$.)  
Let $x \in Bd_{v_{\msf{A}}} \, o(\msf{A},v_{0}) \subset S \subset S^{k}$. Then
    \begin{multline*}
      s^{2} = |x-c|^{2} = \bigl| (x-p)+(p-c) \bigr|^{2} 
        = |x-p|^{2} + 2(x-p) \cdot (p-c) + |p-c|^{2} \\
          \leq r^{2} + 2x \cdot (p-c) - 2p \cdot (p-c) + |p-c|^{2} .
    \end{multline*}
Thus, 
    \begin{equation} \label{E:2x.dot.(p-c).geq}
      2x \cdot (p-c) \geq (s^{2} - r^{2}) +  2p \cdot (p-c) - |p-c|^{2} .
    \end{equation}
Now, since $c, p \in P$; $c = (\cos \theta_{\msf{A}}) v_{\msf{A}}$; and 
$P = v_{\msf{A}}^{\perp} + c$, we have 
    \begin{equation} \label{E:v,c.perp.p-c}
      v_{\msf{A}}, c \perp p-c .
    \end{equation}
Therefore,
    \begin{multline} \label{E:2p.dot.(p-c)-|p-c|.sqrd}
       2p \cdot (p-c) - |p-c|^{2} = 2 \bigl[ (p-c) + c \bigr] \cdot (p-c) - |p-c|^{2} \\
         = 2 |p-c|^{2} + 0 - |p-c|^{2} = |p-c|^{2} > 0.
    \end{multline}
In addition, $s^{2} - r^{2} > 0$. Therefore, plugging \eqref{E:2p.dot.(p-c)-|p-c|.sqrd}  into \eqref{E:2x.dot.(p-c).geq}, we get
    \begin{equation*} 
      2x \cdot (p-c) \geq (s^{2} - r^{2}) + |p-c|^{2} > 0 ,
        \text{ if } x \in Bd_{v_{\msf{A}}} \, o(\msf{A},v_{0}) .
    \end{equation*}
Thus, we may pick $\eta > 0$ s.t.\ 
    \begin{equation} \label{E:x.dot.(p-c).>eta}
      x \cdot (p-c) > \eta > 0 \text{ for every } 
        x \in Bd_{v_{\msf{A}}} \, o(\msf{A},v_{0}) \subset S^{k} .
    \end{equation} 
In particular, $Bd_{v_{\msf{A}}} \, o(\msf{A},v_{0})$ lies in an open hemisphere of $S$. 

Let $K := \bigl\{ x \in S^{k} : x \cdot (p-c) \leq \eta \bigr\}$. Then $K$ is compact and, by \eqref{E:x.dot.(p-c).>eta}, if $x \in K \cap o(\msf{A},v_{0})$, then $x \notin Bd_{v_{\msf{A}}} \, o(\msf{A},v_{0})$ so, by \eqref{E:o(A,v0).subset.X(vA,theta.v0)} and \eqref{E:Bd.o(A).defn}, $x \cdot v_{\msf{A}} > \cos \theta_{\msf{A}}$. 
Since $o(\msf{A},v_{0}) \subset S^{k}$ is also compact, the set 
$K \cap o(\msf{A},v_{0})$ is compact.  
Therefore, we can choose $\beta > 0$ s.t.\ 
    \begin{equation}  \label{E:x.dot.v.when.x.dot.(p-c).leq.eta}
      x \cdot v_{\msf{A}} \geq \cos \theta_{\msf{A}} + \beta 
        \text{ for every } x \in o(\msf{A},v_{0}) \text{ with } x \cdot (p-c) \leq \eta .
    \end{equation} 
Consequently, if $u \in S^{k}$ and $x \in o(\msf{A},v_{0})$ with $x \cdot (p-c) \leq \eta$, 
    \begin{multline}  \label{E:x.dot.u.when.x.dot.(p-c).leq.eta}
      x \cdot u = x \cdot \bigl[ v_{\msf{A}} - (v_{\msf{A}}- u) \bigr] 
        > \cos \theta_{\msf{A}} + \beta - x \cdot (v_{\msf{A}} - u) , \\
        > \cos \theta_{\msf{A}} + \beta - |x||v_{\msf{A}} - u|
           = \cos \theta_{\msf{A}} + \beta - |v_{\msf{A}} - u| , \\
           \text{ for every } x \in o(\msf{A},v_{0}) \text{ with } x \cdot (p-c) \leq \eta 
             \text{ and } u \in S^{k} .
    \end{multline}

Recall that $p \neq c$. Let 
    \begin{equation} \label{E:w=|p-c|.invrs.times.(p-c)}
      w = |p-c|^{-1} (p-c) \in S^{k} .
    \end{equation} 
So, by \eqref{E:v,c.perp.p-c}, 
    \begin{equation}   \label{E:w.perp.vA}
      w \perp v_{\msf{A}} .
    \end{equation} 
Now suppose 
    \begin{equation*}
      x \in o(\msf{A},v_{0}) \text{ and } x \cdot (p-c) > \eta ,
    \end{equation*} 
but do not require $x \in Bd_{v_{\msf{A}}} \, o(\msf{A},v_{0})$. 
(See \eqref{E:x.dot.(p-c).>eta}.) Then
    \begin{equation*}
      \eta < |p-c| \, x \cdot w .
    \end{equation*}
Thus, 
    \begin{equation} \label{E:x.dot.(p-c).leq.eta:x.dot.w.>.|p|.invrs.etc}
      1 \geq x \cdot w > \gamma := |p-c|^{-1} \eta > 0 , 
        \text{ if } x \cdot (p-c) > \eta .
    \end{equation}

Let $\phi \in [0, \pi]$ and 
    \begin{equation}   \label{E:u(phi).defn}
      u(\phi) := (\cos \phi) v_{\msf{A}} + (\sin \phi) w \in S^{k} .
    \end{equation} 
Thus, $u(0) = v_{\msf{A}}$.
By \eqref{E:o(A,v0).subset.X(vA,theta.v0)} and 
\eqref{E:x.dot.(p-c).leq.eta:x.dot.w.>.|p|.invrs.etc},
    \begin{multline}  \label{E:u(phi).dot.x.when.(p-c).>.eta}
      u(\phi) \cdot x = (\cos \phi) \, v_{\msf{A}} \cdot x 
        + (\sin \phi) \, w \cdot x \geq (\cos \phi) (\cos \theta_{\msf{A}}) 
          + (\sin \phi) \, \gamma, \\
          \text{ for every } x \in o(\msf{A},v_{0}) \text{ with } x \cdot (p-c) > \eta .
    \end{multline}
Define 
    \begin{equation*}
       f(\phi) := (\cos \phi) (\cos \theta_{\msf{A}}) + \gamma \sin \phi .
    \end{equation*}
$f(\phi)$ has positive derivative at $\phi = 0$ and $f(0) = \cos \theta_{\msf{A}}$. Therefore, for small $\phi > 0$, we have $f(\phi) > \cos \theta_{\msf{A}} > 0$ (by \eqref{E:three.thetas} and  \eqref{E:theta.v0.geq.pi/4}). 
By \eqref{E:u(phi).dot.x.when.(p-c).>.eta}, 
    \begin{equation}  \label{E:when.u(phi).dot.x.geq.f(phi)} 
         u(\phi) \cdot x \geq f(\phi) 
           \text{ for every } x \in o(\msf{A},v_{0}) \text{ with } x \cdot (p-c) > \eta .
    \end{equation}

Recall the definition, \eqref{E:x.dot.(p-c).leq.eta:x.dot.w.>.|p|.invrs.etc}, of $\gamma$. Let 
    \begin{equation*}
      g(\phi) := (\cos \phi) (\cos \theta_{\msf{A}} + \beta) -  \eta (\sin \phi)/|p - c|
          = (\cos \phi) (\cos \theta_{\msf{A}} + \beta) - \gamma \sin \phi .
    \end{equation*}
$g(\phi)$ has negative derivative at $\phi = 0$ but 
$g(0) = \cos \theta_{\msf{A}} + \beta > \cos \theta_{\msf{A}}$. Therefore, for small 
$\phi > 0$, we still have $g(\phi) > \cos \theta_{\msf{A}} > 0$. By \eqref{E:u(phi).defn}, 
\eqref{E:x.dot.v.when.x.dot.(p-c).leq.eta}, \eqref{E:w=|p-c|.invrs.times.(p-c)}, and \eqref{E:x.dot.(p-c).leq.eta:x.dot.w.>.|p|.invrs.etc}, we have 
    \begin{multline*}
       u(\phi) \cdot x = (\cos \phi) (v_{\msf{A}} \cdot x) + (\sin \phi) (w \cdot x)
         \geq (\cos \phi) (\cos \theta_{\msf{A}} + \beta) - \gamma \sin \phi \\
           \text{ for every } x \in o(\msf{A},v_{0}) \text{ with } x \cdot (p-c) \leq \eta .
    \end{multline*}
Thus, 
    \begin{equation} \label{E:when.u(phi).dot.x.geq.g(phi)}
      u(\phi) \cdot x \geq g(\phi) 
        \text{ for every } x \in o(\msf{A},v_{0}) 
           \text{ with } x \cdot (p-c) \leq \eta .
    \end{equation}

We may pick $\phi' \in (0, \pi)$ s.t.\ $f(\phi') > \cos \theta_{\msf{A}}$ and $g(\phi') > \cos \theta_{\msf{A}}$. 
Therefore, by \eqref{E:when.u(phi).dot.x.geq.f(phi)} and \eqref{E:when.u(phi).dot.x.geq.g(phi)},
    \begin{equation}  \label{E:u(phi')>cos(thetaA)}
      u(\phi') \cdot x \geq \min \bigl\{ f(\phi') , g(\phi') \bigr\} > \cos \theta_{\msf{A}}
        \text{ \emph{for every} } x \in o(\msf{A},v_{0}) .
    \end{equation} 
Let $\theta' := \arccos \bigl( \min \bigl\{ f(\phi') , g(\phi') \bigr\} \bigr) \in (0, \pi/2]$. 
Then $u(\phi') \cdot x \geq \cos \theta'$ for every $x \in o(\msf{A},v_{0})$.
I.e., $o(\msf{A},v_{0}) \subset X \bigl[ u(\phi'), \theta' \bigr]$. (See \eqref{E:X(v.theta).defn}.) But $0 \leq \theta' < \theta_{\msf{A}}$ by \eqref{E:u(phi')>cos(thetaA)}, contradicting the definitions \eqref{E:three.thetas} and \eqref{E:theta.v0.defn} 
of $\theta_{\msf{A}} := \theta_{min} = \theta_{v_{0}}$. We arrived at this contraction beginning at \eqref{E:Br(p).contains.Bd.o(A)}. We conclude that \eqref{E:Br(p).contains.Bd.o(A)} is false and the lemma is proved.
  \end{proof}

  \begin{proof}[Proof of \eqref{E:T.eps.is.sphere.bndl}.] 
Let $\xi \in G(k,\nvar)$. The ``fiber'' of $\T^{\epsilon}_{(0)}$ over $\xi$ consists of all matrices 
$Y \in \T^{\epsilon}_{(0)}$ s.t.\ $\pi(Y) = \rho(Y) = \xi$. Write elements of the sphere (the fiber) 
$S^{(n-1) k-1}$ as $(n-1) \times k$ matrices. Thus, 
    \begin{equation*}
      M^{(n-1) \times k} \in S^{(n-1) k-1} \text{ if and only if } \| M \| = 1,
    \end{equation*}
where $\|\cdot\|$ is the Euclidean or Frobenius norm (see \eqref{E:matrix.norm}). Let $D \in \mcl{V}_{k}$. (See \eqref{E:Vk.defn}.) By \eqref{E:switch.and.trace}, 
$\| MD \| = trace \, \bigl[ (D^{T} M^{T} M) D \bigr] = trace \, \bigl[ D (D^{T} M^{T} M) \bigr] 
= trace \, M^{T} M = 1$ whenever $M \in S^{(n-1)k-1}$. 
Define $h := h_{D}: U_{D} \times S^{(n-1) k-1} \to \D_{(0)}$ by
	\begin{multline}  \label{E:h(M).defn}
		h_{D}(\zeta, M) :=  
				\begin{pmatrix}
					M \\
					0^{1 \times k}
				\end{pmatrix}
			(D^{k \times \nvar}, 0^{k \times (n-\nvar)}) \; Y_{\zeta} =
				\begin{pmatrix}
					M D \, \Pi_{\zeta} \\
					0^{1 \times k}
				\end{pmatrix}^{n \times \nvar} \in \D_{(0)}, \\
				  \zeta \in U_{D} \text{ and } M \in S^{(n-1) k-1}.
	\end{multline}
Thus, by lemma \ref{L:proj.mat.is.imbedding.of.Grass} or \eqref{E:xi.to.Y.xi.imbed}, 
    \begin{equation}  \label{E:hD.is.smooth}
      h_{D} \text{ is smooth. }
    \end{equation}

Notice that
	\begin{equation}  \label{E:rho.hD.zeta.M.in.zeta}
		\rho \bigl[ h_{D}(\zeta, M) \bigr] \subset \zeta, \text{ for } \zeta \in U_{D}.
	\end{equation}
 \emph{Claim:} If $\zeta \in U_{D}$ and $M \in S^{(n-1) k-1}$, then
	\begin{equation} \label{E:hD.nvr.0}
		\bigl\| h_{D}(\zeta, M) \bigr\|^{2} \geq \frac{1}{2k} > 0.
	\end{equation}
By \eqref{E:basic.props.of. Pi}, we have
	\begin{equation}  \label{E:sqrd.norm.hD.expansion}
		\bigl\| h_{D}(\zeta, M) \bigr\|^{2} 
		     = trace ( M D \, \Pi_{\zeta} \Pi_{\zeta} D^{T} M^{T} ) 
		       = trace ( M D \, \Pi_{\zeta} D^{T} M^{T} ) .
	\end{equation}

By the Singular Value Decomposition (Rao \cite[(v), p.\ 42]{crR73.LinStatInf}), we can write 
$M^{(n-1) \times k} = L^{(n-1) \times k} \Lambda N^{T}$, where $L$ has orthonormal columns, 
$\Lambda^{k \times k}$ is non-negative diagonal, and $N^{k \times k}$ is orthogonal. 
Thus, $M^{T} M = N \Lambda^{2} N^{T}$. 
Let $\lambda_{1}, \ldots, \lambda_{k}$ be the diagonal elements of $\Lambda$. WLOG we may assume $\lambda_{1} \geq \cdots \geq \lambda_{k} \geq 0$. Thus, the eigenvalues of 
$M^{T} M$ are $\lambda_{1}^{2}, \cdots, \lambda_{k}^{2}$. Therefore, by \eqref{E:Frob.norm.in.terms.of.eigvals},  
$1 = \| M \|^{2} = \lambda_{1}^{2} + \cdots + \lambda_{k}^{2}$, so $\lambda_{1} \geq 1/\sqrt{k}$. 

We can write $D \, \Pi_{\zeta} D^{T}  = Q \Sigma Q^{T}$, where $Q^{k \times k}$ is orthogonal and $\Sigma^{k \times k}$ is diagonal with diagonal entries the eigenvalues 
of $D \, \Pi_{\zeta} D^{T}$. Since $\zeta \in U_{D}$, the diagonal entries in $\Sigma$ are all 
$> 1/2$. 

Thus, proving \eqref{E:hD.nvr.0} is a matter of bounding below the trace of 
$L \Lambda N^{T} Q \Sigma Q^{T} N \Lambda L^{T}$. Let $u_{1}^{(n-1) \times 1}$ be the first column of $L$. So $|u_{1}|= 1$. 
Then $\Lambda L^{T} u_{1} = \Lambda (1, 0, ..., 0)^{T} = (\lambda_{1}, 0, ..., 0)^{T}$. 
Let $v^{k \times 1} := N \Lambda L^{T} u_{1}$. Then $|v| \geq 1/\sqrt{k}$. Thus, by Rao \cite[1f.2(i), p.\ 62]{crR73.LinStatInf},
    \begin{equation*}
    	\Bigl| u_{1}^{T} \bigl( M D \, \Pi_{\zeta} D^{T} M^{T} \bigr) u_{1} \Bigr|
          = \Bigl| u_{1}^{T} \bigl( L \Lambda N^{T} Q \Sigma Q^{T} N \Lambda L^{T} \bigr) u_{1} \Bigr| 
            = \bigl| v^{T} ( Q \Sigma Q^{T} ) v \bigr| \geq |v|^{2}/2 \geq \frac{1}{2 k}.
    \end{equation*}
Therefore, by by Rao \cite[1f.2(i), p.\ 62]{crR73.LinStatInf} again, $M D \, \Pi_{\zeta} D^{T} M^{T}$ has at least one eigenvalue no smaller than $\tfrac{1}{2 k}$. The claim \eqref{E:hD.nvr.0} now follows from \eqref{E:sqrd.norm.hD.expansion}.

Finally, define
	\begin{multline*}
		\varphi_{D}(\zeta, M)  := \varphi_{D, \epsilon}(\zeta, M) 
		      := \epsilon \bigl\| h_{D}(\zeta, M) \bigr\|^{-1} h_{D}(\zeta, M) + Y_{\zeta} 
		        \in \D_{(0)}, \\
		          (\zeta \in U_{D}, \; M \in S^{(n-1)k-1}).
	\end{multline*}
By \eqref{E:T.eps.defn} and \eqref{E:rho.hD.zeta.M.in.zeta}, 
$\varphi_{D}(\zeta, M) \in \T^{\epsilon}_{(0)}$. In fact, $\varphi_{D}$ will be the local triviality homeomorphism of the bundle for the neighborhood $U_{D}$ (Spanier \cite[p.\ 90]{ehS66}). So we must show \emph{claim:}
	\begin{equation} \label{E:varphi.is.homeom}
		\varphi_{D} \text{ is a homeomorphism of } U_{D} \times S^{(n-1) k-1}
			\text{ onto } \pi^{-1}(U_{D}) \subset \T_{\epsilon}. 
	\end{equation}
By \eqref{E:hD.nvr.0}, \eqref{E:hD.is.smooth}, and \eqref{E:xi.to.Y.xi.imbed}, $\varphi_{D}$ is continuous, in fact, smooth. 

Moreover, $\varphi_{D}$ is injective. To prove this, we first establish
	\begin{equation}  \label{E:pi.varphi.commutes}
	      \pi \bigl( \varphi_{D}(\zeta, M) \bigr) := \rho\bigl( \varphi_{D}(\zeta, M) \bigr) 
		   = \zeta ,
	\end{equation}
where $\pi$ is the restriction, $\pi := \rho \restriction_{\T^{\epsilon}_{(0)}}$. (Note that \eqref{E:pi.varphi.commutes} is one of the requirements 
for $(\T^{\epsilon}_{(0)}, G(k,\nvar), S^{(n-1)k-1}, \pi)$ to be a fiber bundle, 
Spanier \cite[p.\ 90]{ehS66}.) By \eqref{E:rho.hD.zeta.M.in.zeta}, 
we have $\rho\bigl( \varphi_{D}(\zeta, M) \bigr) \subset \zeta$. Conversely, 
since $\varphi_{D}(\zeta, M) \in \T^{\epsilon}_{(0)}$, by \eqref{E:T.eps.in.P.0} and 
\eqref{E:Pf.(0).defn}, for some $\omega \in G(k, \nvar)$, we have 
    \begin{equation*}
        G(k,\nvar) \ni \omega := \rho\bigl( \varphi_{D}(\zeta, M) \bigr) \subset \zeta 
          \in G(k, \nvar).
    \end{equation*}
But $\omega$ and $\zeta$ are both in $G(k, \nvar)$ we must have $\omega = \zeta$. This proves \eqref{E:pi.varphi.commutes}.

Since $\varphi_{D}$ is smooth, by \eqref{E:rho.is.smooth}, the map 
$(\zeta, M) \mapsto \rho\bigl( \varphi_{D}(\zeta, M) \bigr) \in G(k, \nvar)$ is smooth. 
Thus, from $\varphi_{D}(\zeta, M)$, we can smoothly determine $\zeta$ and, hence, by \eqref{E:xi.to.Y.xi.imbed}, smoothly determine the matrix
	\begin{equation*}
	     \epsilon \bigl\| h_{D}(\zeta, M) \bigr\|^{-1} h_{D}(\zeta, M) 
	         = \varphi_{D}(\zeta, M) - Y_{\zeta}.
	\end{equation*}
Hence, by \eqref{E:h(M).defn}, up to a multiplicative constant, we can determine $M D \, \Pi_{\zeta}$ in a smooth fashion from $\varphi_{D}(\zeta,M)$. By \eqref{E:D.Pi.DT.is invrtble}, we have 
	\begin{equation*}
		M = M (D \, \Pi_{\zeta} D^{T}) (D \, \Pi_{\zeta} \, D^{T})^{-1}
		  = (M D \, \Pi_{\zeta}) D^{T} (D \, \Pi_{\zeta} \, D^{T})^{-1}.
	\end{equation*}
Thus, from $\varphi_{D}(\zeta, M)$ we can determine $M D \, \Pi_{\zeta}$ and, by \eqref{E:pi.varphi.commutes}, $\zeta$. Therefore, since $D$ is known, we can determine $D^{T} (D \, \Pi_{\zeta} \, D^{T})^{-1}$, and hence $M$, up to a multiplicative constant. The multiplicative constant can be eliminated because we know $\| M \| = 1$. This completes the proof that $\varphi_{D}$ is injective. Moreover, the operations involved in inverting 
$\varphi_{D}(\zeta, M)$ are continuous, in fact, smooth. Hence, $\varphi_{D}$ is a homeomorphism and it is an immersion as well. 
So by definition, specifically Boothby \cite[Definition (4.11), p.\ 73]{wmB75}, $\varphi_{D}$ is an imbedding. 

Finally, we prove that $\varphi_{D}$ maps $U_{D} \times S^{(n-1)k-1}$ onto 
$\pi^{-1}(U_{D}) \subset \T^{\epsilon}_{(0)}$. Let $Y^{n \times \nvar} \in \pi^{-1}(U_{D})$. Then, by \eqref{E:T.eps.defn}, \eqref{E:T.eps.in.P.0}, and \eqref{E:Pf.(0).defn}, there exists 
$\zeta \in U_{D}$ s.t.\ $\rho(Y) = \zeta$ and $\| Y - Y_{\zeta} \| = \epsilon$. We must have $\rho(Y - Y_{\zeta}) \subset \zeta$ and the last row of $Y - Y_{\zeta}$ must be $0^{1 \times \nvar}$. Look just above \eqref{E:D.has.orthon.rows} for the definition of 
$O_{\xi} \subset \mcl{V}_{k}$ for $\xi \in G(k, \nvar)$. Then, there exists 
$M_{Y}^{(n-1) \times k}$ and $D_{Y}^{k \times \nvar} \in O_{\zeta}$ s.t.\
	\begin{equation}   \label{E:Y.in.pi.invrs.UD}
		Y = \epsilon
 				\begin{pmatrix}
					M_{Y} \\
					0^{1 \times k}
				\end{pmatrix}
			D_{Y} + Y_{\zeta}.
	\end{equation}
We must also have 
    \begin{multline}  \label{E:||MY.D.Y||=1}
      \| Y - Y_{\zeta} \|^{2} = \epsilon^{2} = \epsilon^{2} \| M_{Y} D_{Y} \|^{2} 
       = \epsilon^{2} trace \, (M_{Y} D_{Y} D_{Y}^{T} M_{Y}^{T}) \\
        = \epsilon^{2} trace \, (M_{Y} M_{Y}^{T}) = \epsilon^{2} \| M_{Y} \|^{2} ,
    \end{multline} 
by \eqref{E:D.has.orthon.rows}. I.e., $M_{Y} \in S^{(n-1)k-1}$. By \eqref{E:rho.is.smooth} and lemma \ref{L:proj.mat.is.imbedding.of.Grass} again, the map 
$Y \mapsto \Pi_{\rho(Y)}$ is smooth. Therefore, by \eqref{E:Y.xi.defn}, we can determine $Y_{\zeta}$ smoothly from $Y$. Hence, we can determine $M_{Y} D_{Y}$ smoothly from $Y$. 

But, by \eqref{E:D.has.orthon.rows} again and \eqref{E:D.Pi.DT.is invrtble}, 
$D (D_{Y}^{T} D_{Y} ) D^{T} =D \, \Pi_{\zeta} \, D^{T}$ is invertible so $D D_{Y}^{T}$ is invertible and
	\begin{multline}  \label{E:MY.DY}
		M_{Y} D_{Y} = M_{Y} (D D_{Y}^{T})^{-1} (D D_{Y}^{T}) D_{Y} 
			= M_{Y} (D D_{Y}^{T})^{-1} D (D_{Y}^{T} D_{Y}) \\
			  = \bigl[ M_{Y} (D D_{Y}^{T})^{-1} \bigr] D \Pi_{\zeta}.
	\end{multline}
Thus, by \eqref{E:Y.xi.defn}, 
	\begin{equation*}
		\epsilon
 				\begin{pmatrix}
					M_{Y} \\
					0^{1 \times k}
				\end{pmatrix}
			D_{Y} = \epsilon
 				\begin{pmatrix}
					M_{Y} (D D_{Y}^{T})^{-1} \\
					0^{1 \times k}
				\end{pmatrix}
			(D, 0^{k \times (n - \nvar)})Y_{\zeta}.
	\end{equation*}

That $M_{Y} \in S^{(n-1)k-1}$ means, \emph{a fortiori}, that $M_{Y} \neq 0$, so there exists $a^{k \times 1}$ s.t.\ $M_{Y} a \neq 0$. Hence, $\bigl[ M_{Y} (D D_{Y}^{T})^{-1} \bigr] (D D_{Y}^{T}) a = M_{Y} a \neq 0$. Therefore, $M_{Y} (D D_{Y}^{T})^{-1} \neq 0$ so $\bigl\| M_{Y} (D D_{Y}^{T})^{-1} \bigr\| > 0$. Let 
$M^{(n-1) \times k} := \bigl\| M_{Y} (D D_{Y}^{T})^{-1} \bigr\|^{-1} M_{Y} (D D_{Y}^{T})^{-1} 
\in S^{(n-1)k-1}$. Then, by \eqref{E:MY.DY},
    \begin{equation*}
        h_{D}(\zeta, M)  =  \bigl\| M_{Y} (D D_{Y}^{T})^{-1} \bigr\|^{-1}
			\begin{pmatrix}
				M_{Y} (D D_{Y}^{T})^{-1} D \Pi_{\zeta} \\
				0^{1 \times k}
			\end{pmatrix}
		=  \bigl\| M_{Y} (D D_{Y}^{T})^{-1} \bigr\|^{-1}
			\begin{pmatrix}
				M_{Y} D_{Y} \\
				0^{1 \times k}
			\end{pmatrix} .
    \end{equation*}
In particular, since $\| M_{Y} D_{Y} \| = 1$ by \eqref{E:||MY.D.Y||=1}, 
we have $\| h_{D}(\zeta, M) \| = \bigl\| M_{Y} (D D_{Y}^{T})^{-1} \bigr\|^{-1}$. Therefore, by \eqref{E:Y.in.pi.invrs.UD},
    \begin{multline*}
        \varphi_{D}(\zeta, M) = \epsilon \| h_{D}(\zeta, M) \|^{-1} h_{D}(\zeta, M) + Y_{\zeta}  \\
			= \epsilon \bigl\| M_{Y} (D D_{Y}^{T})^{-1} \bigr\|
			  \bigl\| M_{Y} (D D_{Y}^{T})^{-1} \bigr\|^{-1} 
            			\begin{pmatrix}
            				M_{Y} D_{Y} \\
            				0^{1 \times k}
            			\end{pmatrix}
			+ Y_{\zeta}  
			  = \epsilon
            			\begin{pmatrix}
            				M_{Y} D_{Y} \\
            				0^{1 \times k}
            			\end{pmatrix}
			+ Y_{\zeta} 
			= Y,
    \end{multline*}
as desired. This proves that $\varphi_{D}$ maps $U_{D} \times S^{(n-1)k-1}$ onto $\pi^{-1}(U_{D})$ and in turn completes the proof of the claim \eqref{E:varphi.is.homeom}. Note that, in fact, we have proved that $\varphi_{D} : U_{D} \times S^{(n-1)k-1} \to \pi^{-1}(U_{D})$ is a diffeomorphism.

It further follows from \eqref{E:varphi.is.homeom} and \eqref{E:pi.varphi.commutes} that \eqref{E:T.eps.is.sphere.bndl} holds.
  \end{proof}

  \begin{proof}[Proof that \eqref{E:nontriv.r-dim.homol} holds in the context of section \ref{SS:gnrl.lwr.bnd.plane.fit}.]
By \eqref{E:when.Phi.=.Delta.on.T.eps} and \eqref{E:convergence.in.Grassmann},  
$\Phi  \restriction_{\T^{\epsilon}_{(0)}}$ has a unique continuous extension, $\Theta$, to all of $\T^{\epsilon}_{(0)}$, viz.
	\begin{equation*}
		\Theta(Y) := \pi(Y) := \rho(Y) \in G(k, \nvar), \quad Y \in \T^{\epsilon}_{(0)}.
	\end{equation*}
We prove that $\Theta$ satisfies \eqref{E:nontriv.r-dim.homol}. I.e., we need to show that $\Theta_{\ast}$ is a nontrivial homomorphism 
in dimension $r=1$. 

Let 
	\begin{equation*}
		L := (n-1)k - 1 \geq 1.
	\end{equation*}
We give two approaches, both suggested by Steven Ferry (personal communication). 
They are both based on the fact that, by \eqref{E:T.eps.is.sphere.bndl}, 
$E := \T^{\epsilon}_{(0)}$ is the total space of an $L$-sphere bundle over 
$B := G(k,\nvar)$ with bundle map $\pi := \rho$ and, 
by Milnor and Stasheff \cite[Problem 7-B, p.\ 87]{jwMjdS74}, 
$H^{1}(B; \ZZ/2) = H^{1}\bigl( G(k,\nvar); \ZZ/2 \bigr)$ is non trivial. 

\emph{Method 1. Lift a cycle:} This is an elaboration of a suggestion by Steven Ferry. Since $H^{1}\bigl( B; \ZZ/2 \bigr)$ is nontrivial, by Munkres \cite[Corollary 53.6, p.\ 326]{jrM84}), we have that $H_{1}\bigl( B; \ZZ/2 \bigr)$ is nontrivial as well. Let $z = T_{0} + \cdots + T_{m}$ be a singular 1-cycle in $S_{1}(B) \otimes \ZZ/2$ (where $S_{1}(B)$ is the group of singular 1-chains in $B$) representing a nontrivial singular homology class in $H_{1}(B; \ZZ/2)$. We may assume the $T_{i}$'s are distinct. Let $\Delta_{1}$ be the standard 1-simplex (Munkres \cite[p.\ 162]{jrM84}), so $T_{i}$ is a continous map from $\Delta_{1}$ into $B$. $\Delta_{1}$ is homeomorphic to the unit interval, $I$. Then each $T_{i}$ can be regarded as a continuous map $I \to B$. 

\emph{Claim:} Relabeling if necessary,  for some $m' = 0, \ldots, m$ there is a non-bounding cycle $z' := T_{0} + \cdots + T_{m'} \in S_{1}(B) \otimes \ZZ/2$ 
in which $T_{i}(1) = T_{i+1}(0)$ ($i=0, \ldots, m'$), where $m'+1$ is identified with $0$. This, we claim, is true for any singular 1-cycle of $m$ terms. (But $m'$ is not constant.)
\emph{Pf:} We have 
    \begin{equation*}
        0 = \partial z = \sum_{i=0}^{m} \bigl( T_{i}(0) + T_{i}(1) \bigr).
    \end{equation*}
Thus, for every $i = 0, \ldots, m$ and $\alpha = 0$ or 1 there exists $j = 0, \ldots, m$ and $\beta = 0$ or 1 s.t.\ $(i, \alpha) \neq (j, \beta)$ but $T_{i}(\alpha) = T_{j}(\beta)$. Therefore, if $m = 0$ we may take $z' :=T_{0}$. I.e., the claim holds with $m' = m = 0$. 

Suppose the claim holds for $m = n$ for some $n = 0, 1, 2, \ldots$ and let $m = n+1$. 
First, suppose $T_{0}(1) = T_{0}(0)$ so $T_{0}$ is a cycle. If $T_{0}$ does not bound then take $z' = T_{0}$. If $T_{0}$ does bound then $z$ is homologous 
to $z' := T_{1} + \cdots + T_{m}$ so the claim holds by the induction hypothesis with 
$m = n$.

So suppose $T_{0}(1) \neq T_{0}(0)$. Pick $j > 0$ s.t.\ $T_{0}(1) = T_{j}(0)$ or $T_{0}(1) = T_{j}(1)$. In the latter case we may reverse $T_{j}$ so that $T_{0}(1) = T_{j}(0)$. Renumbering if necessary, we 
may assume $j = 1$. Let $T_{0}'$ be the singular simplex obtained by concatenating $T_{0}$ and $T_{1}$. Then $T_{0}, T_{1}, T_{0}'$ are the faces of an, admittedly degenerate, singular 2-simplex in $S_{1}(B) \otimes \ZZ/2$. Therefore, $T_{0} + T_{1} + T_{0}'$ is homologous to 0. This means $z$ is homologous to 
$z' := T_{0}' + T_{2} + \cdots + T_{m}$, so $z'$ does not bound. 

By the induction hypothesis, for some $m' = 0, \ldots, m$ there is a non-nonbounding cycle 
$z'' := \tilde{T}_{0} + \tilde{T}_{1} + \cdots + \tilde{T}_{m'}$, consisting of distinct terms, with the following properties. 
First, $\tilde{T}_{i}(1) = \tilde{T}_{i+1}(0)$ ($i=1, \ldots, m'$), where $m'+1$ is identified with $0$. Second, one of the following two statements holds. One possibility is that for every $i = 0, \ldots, m'$ there is some $\ell(j) = 0, \ldots, m$ with 
$\tilde{T}_{j} = T_{\ell(j)}$. In that case we are done. 

Otherwise, for some $i = 1, \ldots, m'$ we have $\tilde{T}_{i} = T_{0}'$ 
and for $j = 0, \ldots, m'$ with $j \neq i$ there is some $\ell(j) = 0, \ldots, m$ 
with $\tilde{T}_{j} = T_{\ell(j)}$. Now, $z''$ is homologous to 
$z' := (T_{0} + T_{1}) + \sum_{j \neq i} T_{\ell(j)}$. After relabeling, $z'$ has the desired properties. This proves the claim. Replace $z$ by $z'$ and $m$ by $m'$. Having redefined $z$, if necessary, in this way, we have $T_{i}(1) = T_{i+1}(0)$ ($i=0, \ldots, m$).

Let $i = 0, \ldots, m$. Now, if $P$ is an arbitrary one point space, we can also regard 
$T_{i}$ as a map 
from $P \times I$ into $B := G(k,\nvar)$. I.e., we may think of $T_{i}$ as a homotopy between the maps $P \mapsto T_{i}(0) \in B$ and $P \mapsto T_{i}(1) \in B$. 
Let $\zeta_{ij} := T_{i}(j) \in B$ ($j=0,1$), so $\zeta_{m1} = \zeta_{00}$.
Let $Y_{i0} \in \pi^{-1}(\zeta_{i0}) \in \T^{\epsilon}_{(0)}$ satisfy 
$\pi(Y_{i0}) = \zeta_{i0}$. $Y_{i0}$ will be specified later. 
Let $f_{i0} : P \times \{0\} \to \T^{\epsilon}_{(0)}$ be the constant map 
$f_{i0}(P,0) := Y_{i0}$. 
By \eqref{E:T.eps.is.sphere.bndl} and Spanier \cite[Corollary 14,p.\ 96]{ehS66}, 
$\pi : \T^{\epsilon}_{(0)} \to G(k,n)$ has the homotopy lifting property. Therefore, there exists a singular 1-simplex $T^{\epsilon}_{i} : I \to \T^{\epsilon}_{(0)}$ s.t.\ 
$T^{\epsilon}_{i}(0) = Y_{i0}$ and $\pi \circ T^{\epsilon}_{i} = T_{i}$.

Let $Y_{00} \in \T^{\epsilon}_{(0)}$ be arbitrary with $\pi(Y_{00}) = \zeta_{00}$. Let $T^{\epsilon}_{0}$ be constructed as above. 
Given $T^{\epsilon}_{i}$, let $Y_{(i+1)0} := T^{\epsilon}_{i}(1)$ and construct $T^{\epsilon}_{i+1}$ as above ($i = 0, \ldots, m-1$). 

We then have $\pi(Y_{m1}) = \zeta_{m1} = \zeta_{00} = \pi(Y_{00})$, but it may not be the case that $Y_{m1} = Y_{00}$. However, $Y_{m1}$ and $Y_{00}$ are both in $\pi^{-1}(\zeta_{m1}) = \pi^{-1}(\zeta_{00})$ which is homeomorphic to the sphere $S^{L}$. Since $L \geq 1$, there is an arc in $\pi^{-1}(\zeta_{00})$ joining $Y_{m1}$ and $Y_{00}$. 
Let $T^{\epsilon}_{m+1} : I \to \pi^{-1}(\zeta_{00})$ be that arc. 
Thus, $T^{\epsilon}_{m+1}$ is a singular 1-simplex and 
$T^{\epsilon}_{i}(1) = Y_{(i+1)0} = T^{\epsilon}_{i+1}(0)$ ($i=0, \ldots, m+1$), where $m+2$ is identified with $0$. Therefore, $z^{\epsilon} := T^{\epsilon}_{0} + \cdots + T^{\epsilon}_{m} + T^{\epsilon}_{m+1}$ is a singular 1-cycle 
in $S(\T^{\epsilon}_{(0)}) \otimes \ZZ/2$. Moreover, $\pi_{\#}(z^{\epsilon})$ is homologous to $z$, where $\pi_{\#} :  S_{1}(\T^{\epsilon}_{(0)}) \otimes \ZZ/2 \to S(B) \otimes \ZZ/2$ is the singular chain map corresponding to $\pi$. 
($\pi_{\#}(z^{\epsilon}) - z$ is a constant singular 1-simplex.)
Thus, $\pi_{\ast}\bigl( \{z^{\epsilon}\} \bigr) = \{z\} \neq 0$. 
This shows that $\Theta_{\ast} := \pi_{\ast}$ is nontrivial. I.e., \eqref{E:nontriv.r-dim.homol} holds.

\emph{Method 2. Vietoris-Begle theorem:} 
As observed near the beginning of chapter \ref{Chptr:sings.in.plane.fit}, $B := G(k,\nvar)$ is compact. By \eqref{E:T.eps.is.manif}, so is $\T^{\epsilon}_{(0)}$. 
By \eqref{E:rho.surj.on.T.eps}, $\pi : \T^{\epsilon}_{(0)} \to B$ is surjective. 
By compactness, $\pi$ is closed. Recall the definition, \eqref{E:pi.is.restrict.of.rho.to.T.eps}, of $\pi$. By \eqref{E:T.eps.is.sphere.bndl}, for every $\xi \in B$, we have that 
$\pi^{-1}(\xi) \homeomto S^{L}$ so $\tilde{H}^{\ell} \bigl[ \pi^{-1}(\xi); \ZZ/2  \bigr] = 0$ 
for $\ell = 0, \ldots, L-1 \geq 0$. Hence, by the Vietoris-Begle theorem (Spanier \cite[Theorem 15, p.\ 344]{ehS66}, Bredon \cite[Theorem 6.1, p.\ 318]{geB97.SheafThy}), 
$\pi^{\ast} : H^{1} ( B; \ZZ/2) \to H^{1} ( \T^{\epsilon}_{(0)}; \ZZ/2)$ is a monomorphism. 
By Munkres \cite[Corollary 53.6, p.\ 326]{jrM84}, we have the commutative diagram 
   \begin{equation*} 
      \begin{CD}
        \text{Hom}_{\ZZ/2} \bigl( H_{1}(B), \ZZ/2 \bigr) @<{\isomto}<< H^{1}(B) \\
        @V{Hom_{\ZZ/2}(\pi_{\ast})}VV       @VV{\pi^{\ast}}V \\
        \text{Hom}_{\ZZ/2} \bigl( H_{1}(\T^{\epsilon}_{(0)}), \ZZ/2 \bigr) 
            @<{\isomto}<< H^{1}(\T^{\epsilon}_{(0)}) .
      \end{CD}
   \end{equation*}
Here, (co)homology is computed with $\ZZ/2$ coefficients. it follows that 
$Hom_{\ZZ/2}(\pi_{\ast})$ is a monomorphism. As observed above, $H^{1}(B; \ZZ/2)$ is nontrivial. Therefore, $Hom_{\ZZ/2}(\pi_{\ast})$ is nonzero. That means 
$\Theta_{\ast} = \pi_{\ast}$ is nontrivial in dimension $r = 1$ and \eqref{E:nontriv.r-dim.homol} holds. (See \eqref{E:r=1}.) 
  \end{proof}
	
  \begin{proof}[Proof of \eqref{E:geod.on.prod.of.spheres}]
This is obvious but, as an exercise, let's prove it! Suppose $\gamma_{i}$ 
($i \in \NN_{n}$) are as in \eqref{E:geod.on.prod.of.spheres}. 

First, we prove the well-known fact (Boothby \cite[pp.\ 307--308]{wmB75}) that a curve $\gamma_{1}$ in a single sphere $S^{\nvar}$, that parametrizes a great circular arc on $S^{\nvar}$ w.r.t.\ a parameter proportional to arc length, is a geodesic on $S^{\nvar}$, and only such $\gamma_{1}$ is a geodesic. By Boothby \cite[Definition (5.1), p.\ 326]{wmB75}, $\gamma : (a,b) \to \D$ is a geodesic on $\D$ if and only if $\tfrac{D}{dt} \tfrac{d \gamma}{dt} = 0$ for all $t \in (a,b)$. We compute $\tfrac{D}{dt} \tfrac{d \gamma_{1}}{dt}$. This requires Christoffel symbols in Boothby \cite[Corollary (3.8), p.\ 318]{wmB75} or Spivak \cite[p.\ 210]{mS79.SpivakVol2}:
    \begin{equation}  \label{E:Christ.awful.symbol}
      \Gamma_{ij}^{k} 
        := \frac{1}{2} \sum_{s=1}^{\nvar} g^{ks} \left( \frac{\partial g_{si}}{\partial x_{j}} 
          -  \frac{\partial g_{ij}}{\partial x_{s}} +  \frac{\partial g_{sj}}{\partial x_{i}} \right), 
    \end{equation}
where the various $g$'s and $x$'s will be explained below. Let $Y(t)$ be a vector field along $\gamma_{1}$. Then from Boothby \cite[(3.11), p.\ 319]{wmB75}:
    \begin{equation}  \label{E:covar.deriv.gnrl.defn}
      \frac{DY}{dt} = \sum_{k=1}^{\nvar} 
        \left( \frac{d b_{k}}{dt} + \sum_{i,j=1}^{\nvar} \Gamma_{ij}^{k} 
          \bigl[ \gamma_{1}(t) \bigr] b_{i}(t) \tilde{x}_{j}'(t) \right) E_{k} , 
    \end{equation}
where $b = (b_{1}, \ldots, b_{\nvar})$ and $\tilde{x} = (x_{1}, \ldots, x_{\nvar})$ will be explained below.

Tentatively define
    \begin{equation} \label{E:gamma.1.curve}
     \gamma_{1}(t) = ( \cos t, 0, \ldots, 0, \sin t ) \in \RR^{\nvar+1}, 
       \qquad 0 < t_{0} < t < t_{1} < \pi ,
    \end{equation}
where $t_{0}, t_{1} $ are fixed. We verify that $\gamma_{1}$ defined this way is a geodesic on $S^{\nvar}$.

Parametrize the upper open hemisphere, $H$, of $S^{\nvar}$ by: 
    \begin{equation*}
      \zeta : (x_{1}, \ldots, x_{\nvar}) \mapsto \left( x_{1}, \ldots, x_{\nvar}, 
        \sqrt{1 - x_{1}^{2} - \cdots - x_{\nvar}^{2}} \right)  \in \RR^{\nvar+1}, 
          \qquad (x_{1}, \ldots, x_{\nvar}) \in B_{1}^{\nvar}(0) .
    \end{equation*}
(See \eqref{E:Euc.ball.defn}.)  Thus,  
    \begin{equation} \label{E:gamma1.in.local.coords}
      \gamma_{1}(t) = \zeta(\cos t, 0, \ldots, 0) .
    \end{equation}

Write $x^{1 \times (\nvar+1)} = \bigl( \tilde{x}^{1 \times \nvar}, x_{\nvar+1} \bigr) 
= \zeta(x_{1}, \ldots, x_{\nvar})$. (Thus, we write the individual coordinates without the tilde.) Thus, in local coordinates, 
$\gamma_{1}(t)$ is $\tilde{x}(t) = (\cos t, 0, \ldots, 0) \in B_{1}^{\nvar}(0)$ and
    \begin{equation}  \label{E:deriv.gamma1.in.loc.coords}
      \text{In local coordinates } \frac{d}{dt} \tilde{x}(t) = \tilde{x}'(t) 
        = (-\sin t, 0, \ldots, 0) \in B_{1}^{\nvar}(0) .
    \end{equation}

Let $E_{i, \zeta(\tilde{x})} := \zeta_{\ast} (\partial/\partial u_{i} \restriction_{u=\tilde{x}})$ 
($i=1, \ldots, \nvar; \tilde{x} \in B_{1}^{\nvar}(0)$), be the coordinate frame on $H$. 
Now, $\partial/\partial y_{i}$, ($i = 1, \ldots, \nvar+1$) span the tangent space 
of $\RR^{\nvar+1}$. By Boothby \cite[Theorem (1.6), p.\ 109]{wmB75}, we have
    \begin{equation}  \label{E:Ei.on.sphere}
         E_{i,\tilde{x}} = \left( \frac{\partial}{\partial y_{i}} 
           \restriction_{y=\tilde{x}} \right) - \left( \frac{x_{i}}{x_{q+1}} 
             \frac{\partial}{\partial y_{\nvar+1}} \restriction_{y=\tilde{x}} \right), 
               \quad i = 1, \ldots, \nvar; \, \tilde{x} \in H .
    \end{equation}
 
Write $\gamma_{1} = (\gamma_{1,1}, \ldots, \gamma_{1, \nvar+1})$, 
so, e.g., $\gamma_{1,1}(t) = \cos(t)$ 
and $\gamma_{1, \nvar+1}(t) = \sin t$. If $\tilde{x} \in B_{1}^{\nvar}(0)$, 
write $x_{\nvar+1} := \sqrt{1 - |\tilde{x}|^{2}}$, 
so $\zeta(\tilde{x}) = (\tilde{x}, x_{\nvar+1})$.
Let 
$Y_{\gamma(t)} := \dot{\gamma}_{1}(t) = \gamma_{1\ast}(d/ds \restriction_{s=t})$ and write $Y(t) := Y_{\gamma(t)} = \sum_{i} b_{i}(t) E_{i}$. 
By \eqref{E:Ei.on.sphere} and Boothby \cite[Theorem (1.6), p.\ 109]{wmB75} again,  
    \begin{multline} \label{E:the.bi's}
      \gamma_{1 \ast} \left( \frac{d}{dt} \right) 
        = - \sin t \frac{\partial}{\partial y_{1}} + \cos t \frac{\partial}{\partial y_{\nvar+1}} 
          =  - (\sin t) E_{1} , \\ 
            \text{ so } b_{i}(t) =
          \begin{cases} 
               - \sin t, &\text{ if } i = 1, \\     
               0, &\text{ otherwise }.
          \end{cases}
    \end{multline}

Plugging this and \eqref{E:deriv.gamma1.in.loc.coords} into \eqref{E:covar.deriv.gnrl.defn}, we get
    \begin{equation} \label{E:covar.deriv.of.gamma.ast}
           \frac{DY}{dt} =  - (\cos t) E_{1} + \sum_{k=1}^{\nvar} \Gamma_{11}^{k} 
              \bigl[ \gamma_{1}(t) \bigr] (\sin^{2} t) E_{k} .
    \end{equation}

It remains to evaluate $\Gamma_{11}^{k} \bigl[ \gamma_{1}(t) \bigr]$ 
($k=1, \ldots, \nvar$). Let $\tilde{x} = (x_{1}, \ldots, x_{\nvar}) \in  B_{1}^{\nvar}(0)$ and let 
$y := \zeta(\tilde{x}) \in H$. By \eqref{E:Ei.on.sphere} and \eqref{E:directional.location.Riem.metric}, 
    \begin{multline}  \label{E:Riemannian.gij.on.sphere}
      g_{ij}(y):= E_{i,y} \cdot E_{j,y} = 
          \begin{cases} 
               1 + \frac{x_{i}^{2}}{x_{\nvar+1}^{2}}, &\text{ if } i = j, \\     
               \frac{x_{i} x_{j}}{x_{\nvar+1}^{2}}, &\text{ otherwise } ,
          \end{cases} \\
               i,j =1, \ldots, \nvar .
    \end{multline}
Let $G_{1}^{\nvar \times \nvar}$ be the matix $(g_{ij})$. 
Thus, 
    \begin{equation*}
      G_{1}^{\nvar \times \nvar} = I_{\nvar} + x_{\nvar+1}^{-2} \; \tilde{x}^{T} \tilde{x}, 
    \end{equation*}
where we view 
$x^{1 \times (\nvar+1)} = (\tilde{x}, x_{\nvar+1}) = (x_{1}, \ldots, x_{\nvar}, x_{\nvar+1})$ as a row vector. 

Let $z := 1/x_{\nvar+1}^{2} = 1/(1 - x_{1}^{2} - \cdots - x_{\nvar}^{2})$. By \eqref{E:Riemannian.gij.on.sphere}, we have 
    \begin{multline}  \label{E:partial.deriv.of.g.mu.nu}
      \frac{\partial}{\partial x_{\lambda}} g_{\mu\nu}(y):= 
          \begin{cases}  
               z^{2} x_{\nu} (2x_{\mu}^{2} + x_{\nvar+1}^{2}), 
                 &\text{ if } \lambda = \mu \neq \nu, \\
               z^{2} x_{\mu} (2x_{\nu}^{2} + x_{\nvar+1}^{2}) , 
                 &\text{ if } \lambda = \nu \neq \mu, \\
               2 z^{2} x_{\mu} (x_{\mu}^{2} + x_{\nvar+1}^{2}), 
                 &\text{ if } \lambda = \mu = \nu, \\
               0, &\text{ otherwise, } 
          \end{cases} \\
          \lambda, \mu, \nu = 1, \ldots, \nvar .
    \end{multline}

Suppose $y = \gamma_{1}(t)$ so 
$y = ( \cos t, 0, \ldots, 0, \sin t) = (x_{1}, \ldots, x_{\nvar}, x_{\nvar+1})$ and $z = \sin^{-2} t$. Note that in this case $x_{1}^{2} + x_{\nvar+1}^{2} = 1$. Then the preceding becomes  
    \begin{multline}  \label{E:partial.deriv.of.g.on.gamma1}
      \frac{\partial}{\partial x_{\lambda}} g_{\mu\nu}(y):= 
          \begin{cases}  
               z \cos t, &\text{ if } \lambda = \mu \neq \nu = 1, \\
               z \cos t, &\text{ if } \lambda = \nu \neq \mu = 1, \\ 
               2 z^{2} \cos t, &\text{ if } \lambda = \mu = \nu = 1, \\
               0, &\text{ otherwise, } 
          \end{cases} \\
          \lambda, \mu, \nu = 1, \ldots, \nvar .
    \end{multline}
Observe that the ``otherwise'' case in the preceding includes 
$\tfrac{\partial}{\partial x_{\lambda}} g_{\mu\nu}(y)$ 
with $\lambda = \mu = 1$, $\nu \neq 1$ and $\lambda = \nu = 1$, $\mu \neq 1$.

Write $G_{1}^{-1} = (g^{ij})$. Then it is easy to see using 
the fact that $|\tilde{x}|^{2} + x_{\nvar+1}^{2} = 1$,  that 
$G_{1}^{-1} = I_{\nvar} - \tilde{x}^{T} \tilde{x}$. Hence, by \eqref{E:gamma.1.curve}, 
at $x = \gamma_{1}(t)$ we have $\tilde{x} = (\cos t, 0, \ldots, 0) \in \RR^{\nvar}$ and
    \begin{equation}   \label{E:g.ks.on.sphere}
     g^{ks} \bigl[ \gamma_{1}(t) \bigr] =
          \begin{cases} 
               1 - \cos^{2} t = \sin^{2} t ,
                 &\text{ if } k = s = 1, \\     
               1, &\text{ if } 1 < k = s \leq \nvar , \\  
               0, &\text{  otherwise } .
          \end{cases}
      \end{equation}  

By \eqref{E:covar.deriv.of.gamma.ast}, we only have to calculate 
$\Gamma_{11}^{k}$. 
From \eqref{E:partial.deriv.of.g.on.gamma1}, we see that 
$\tfrac{\partial g_{s1}}{\partial x_{1}} -  \tfrac{\partial g_{11}}{\partial x_{s}} 
+  \tfrac{\partial g_{s1}}{\partial x_{1}} = 2 z^{2} \cos t$ if $s = 1$. And from the ``otherwise'' case in \eqref{E:partial.deriv.of.g.on.gamma1}, we see that 
$\tfrac{\partial g_{s1}}{\partial x_{1}} -  \tfrac{\partial g_{11}}{\partial x_{s}} 
+  \tfrac{\partial g_{s1}}{\partial x_{1}} = 0$ if $s \neq 1$. 
Thus, by \eqref{E:Christ.awful.symbol}, at $y = \gamma_{1}(t) = ( \cos t, 0, \ldots, 0, \sin t) = (x_{1}, \ldots, x_{\nvar}, x_{\nvar+1})$ we have, 
    \begin{equation*}
      \Gamma_{11}^{k}(y) = 
        \frac{1}{2} \sum_{s=1}^{\nvar} g^{ks} \left( \frac{\partial g_{s1}}{\partial x_{1}} 
              -  \frac{\partial g_{11}}{\partial x_{s}} 
                +  \frac{\partial g_{s1}}{\partial x_{1}} \right) 
                  = g^{k1} z^{2} \cos t .
    \end{equation*}

By \eqref{E:g.ks.on.sphere}, $g^{k1} \bigl[ \gamma_{1}(t) \bigr] = 0$ unless $k = 1$.
Moreover, by the preceding, \eqref{E:g.ks.on.sphere}, and the fact that 
$z = 1/\sin^{2} (t)$,  
    \begin{multline}  \label{E:Christ.awful.sym.at.gamma1}
      \text{When } y = \gamma_{1}(t), \text{ we have } \\
      \Gamma_{11}^{1}(y) = (\sin^{2} t) z^{2} \cos t 
        = \frac{\cos t}{\sin^{2} t} \text{ and }
        \Gamma_{11}^{k}(y) = 0 \text{ for } k > 1.
    \end{multline}
Therefore, by \eqref{E:covar.deriv.of.gamma.ast}, the $E_{1}$ component of $D Y/dt$ is
    \begin{equation*}
       - \cos t + \frac{\cos t}{\sin^{2} t} \sin^{2} t = 0 .
    \end{equation*}
All the other components are 0 as well. Thus, $\gamma_{1}$ is a geodesic.

Now, the formula \eqref{E:gamma.1.curve} for $\gamma_{1}$ makes sense for every $t \in \RR$. Hence, $\gamma_{1}$ is a geodesic as a function on $\RR$. We \emph{claim} that every geodesic on $S^{\nvar}$ when parametrized by arc length can be put in the form $\gamma_{1}(t)$, for $t$ in some interval, by an appropriate choice of coordinates. 

For let $\phi : [c,d] \to S^{\nvar}$ be a geodesic on $S^{\nvar}$ parametrized by arc length. Let $s_{0} \in (c,d)$ and let $x_{0} = \phi(s_{0})$. By Boothby \cite[Theorem (7.2), p.\ 340]{wmB75}, for $s_{1} \in (s_{0},d)$ sufficiently close to $s_{0}$ then 
$\phi \restriction_{I}$, the restriction to the interval $I := (s_{0}, s_{1})$, is the unique geodesic parametrized by arc length that joins $x_{0} := \phi(s_{0})$ to $x_{1} := \phi(s_{1})$ with minimum arc length among all such. 

By \eqref{E:directional.location.Riem.metric}, the Riemannian metric, $g$, on $S^{\nvar}$ is inherited from the inner product on $\RR^{\nvar+1}$. That inner product is invariant under orthogonal transformation. Hence, so is $g$. By an appropriate orthogonal change of coordinates, there is a pair $t_{0}, t_{1} \in [0, \pi]$ s.t.\ $t_{0} < t_{1}$ and
$x_{i} = \gamma_{1}(t_{i})$ ($i=0,1$). By making $s_{1} \in (s_{0},d)$ closer 
to $s_{0}$, and hence $x_{1}$ closer to $x_{0}$, if necessary, by Boothby \cite[Theorem (7.2), p.\ 340]{wmB75} again, we may assume 
$\gamma_{1} \restriction_{J}$, where $J := [t_{0}, t_{1}]$, is the unique geodesic parametrized by arc length that joins $x_{0}$ to $x_{1}$ with minimum arc length among all such. Since $\phi$ is parametrized by arc length, we may assume $s_{i} = t_{i}$. Thus, $\phi \restriction_{I} = \gamma_{1} \restriction_{I}$.

Let $t_{2} := \sup \bigl\{ t \in [t_{1}, d] : \phi(t) = \gamma_{1}(t) \bigr\}$. By continuity, $\phi(t_{2}) = \gamma_{1}(t_{2})$. Suppose $t_{2} < d$. Again by continuity, $\phi_{\ast}( \tfrac{d}{dt} \restriction_{t=t_{2}} ) = \gamma_{1, \ast}( \tfrac{d}{dt} \restriction_{t=t_{2}} )$. 
But by Boothby \cite[Theorem (5.8), p.\ 330]{wmB75}, 
there then exists $\epsilon \in (0, d-t_{2})$ s.t.\ 
$\phi(t) = \gamma_{1}(t)$ for $t \in [t_{2}, t_{2}+\epsilon)$. This contradicts the definition of $t_{2}$. Therefore, $t_{2} = d$, so $\phi(t) = \gamma_{1}(t)$ for $t \in [t_{0}, d]$ A similar argument shows that 
$\phi(t) = \gamma_{1}(t)$ for $t \in [c, t_{0}]$. Thus, $\phi = \gamma_{1} \restriction_{[c,d]}$. This proves the claim that every geodesic on $S^{\nvar}$ when parametrized by arc length can be put in the form $\gamma_{1}$ (restricted, perhaps, to a subinterval of $\RR$) by an appropriate choice of coordinates.

Now consider geodesics on the product $\D := (S^{\nvar})^{n}$. Let $\gamma_{i}$ ($i \in \NN_{n}$) be as in \eqref{E:geod.on.prod.of.spheres} and let 
 $\gamma := \gamma_{1} \times \cdots \times \gamma_{n}$.
By \eqref{E:directional.location.Riem.metric} the matrix of the overall Riemannian, 
$G^{\nvar n \times \nvar n}$, is block diagonal. Each block is of the form $G_{1}$ as in \eqref{E:Riemannian.gij.on.sphere}. 

Implement a system of double indices $(i,j)$ with $i \in \NN_{n}$ indicating block and $j = 1, \ldots, \nvar$ indicating coordinate within the $i^{th}$ $\nvar$-sphere factor. Thus, entries in $G$ are labeled $g_{(i,j),(k,\ell)}$ with $g_{(i,j),(k,\ell)} = 0$ if $i \neq k$ and 
$g_{(i,j),(i,\ell)} = g_{j\ell}$ as in \eqref{E:Riemannian.gij.on.sphere}. Then 
$g^{(i,k), (j,\ell)}$ and 
$\Gamma_{(j,\ell),(t,s)}^{(i,k)}$ have the obvious meanings. Any of these quantities whose sub- or superscripts have first components that do not match is 0. 
E.g., $g^{(i,j),(k,\ell)} = 0$ if $i \neq k$. $\Gamma_{(j,\ell),(t,s)}^{(i,k)} = 0$ if $j, t, i$ are not the same. 
It follows that every geodesic on $\D := (S^{\nvar})^{n}$ is of the form 
$\gamma_{1} \times \cdots \times \gamma_{n}$, where each $\gamma_{i}$ is a great circular arc (appropriately parametrized). 
  \end{proof}
  
    	\begin{proof}[Proof of proposition \ref{P:rate.of.decrease.of.Hm.St.aug.direct.mean}]
Algebraic geometry might be used to prove the proposition. We use elementary methods.

First, we prove \eqref{E:bunch.of.neg.y0s.is.closest.to.T}. Recall that any point $z \in \T$ is of the form $(w, \ldots, w)$, for some $w \in S^{\nvar}$. Let $x \in \Ss_{t}$ and write $x = (y_{1}, \ldots, y_{n})$. By \eqref{E:Sa=Sa'} and \eqref{E:S'a.D'a.defns}, 
$\sum_{i=1}^{n} y_{i} = -t y_{0}$. Thus, by the (Cauchy-)Schwarz inequaltiy,
    \begin{equation}  \label{E:|x-z|^{2}.=.2n+2tw.dot.y.0}
    	|x - z|^{2} = 2n - 2 \sum_{i=1}^{n} y_{i} \cdot w = 2n - 2 w \cdot \sum_{i=1}^{n} y_{i} 
	  = 2n + 2 t \, w \cdot y_{0} \geq 2(n-t).
    \end{equation}
This is minimized, in $w \in S^{\nvar}$, by $w = -y_{0}$, i.e., $z = x_{0}$ and \eqref{E:bunch.of.neg.y0s.is.closest.to.T} is proved. This proves part \ref{I:aug.mean.euc.dist} of the proposition. 

\emph{Next, we prove part \ref{I:aug.mean.geod.dist} of the proposition.} Let $w, y \in S^{\nvar}$. The geodesic distance on $S^{\nvar}$ between $w$ and $y$ is the smaller of the angles between them. (See \eqref{E:angle.between.vectors}.) By definition, that angle is $\leq \pi/2$. Then if $s := |w-y|/2$, then the great circular distance between $w$ and $y$ is clearly $2 \arcsin s$. By \eqref{E:rho.is.top.metric}, $\rho$ is the topological metric 
on $\D$ determined by the Riemannian metric on $\D$. (So $\rho$ plays the role of $\xi$ as defined in \eqref{E:xi.is.metric.on.D}.) Now, let $x = (y_{1}, \ldots, y_{n}) \in \D$ 
and $x' = (y'_{1}, \ldots, y'_{n}) \in \D$ with $y'_{i} \cdot y_{i} \geq 0$ for $i \in \NN_{n}$. Let $s_{i} = |y'_{i} - y_{i}|/2$. Then,  by \eqref{E:geod.on.prod.of.spheres}, 
    \begin{equation}  \label{E:geod.dist.on.product.of.spheres}
        \rho(x,x') = \sqrt{ 4 \sum_{i=1}^{n} \arcsin^{2} s_{i} } .
    \end{equation}
(See \eqref{E:geod.dist.on.D.2} in the $\nvar =1$ case.)
 
We can use \eqref{E:geod.dist.on.product.of.spheres}, 
with $x' = x_{0} = (-y_{0}, \ldots, -y_{0}) $ and 
$s_{i} := |y_{i} + y_{0}|/2 \leq 1$, to compute $\rho(x,x_{0})$. First, we study the behavior of 
    \begin{equation*}
        \sqrt{ \frac{4 \sum_{i=1}^{n} \arcsin^{2} s_{i} }{ 2(n-t) } }
    \end{equation*}
as $x \to x_{0}$. By \eqref{E:|x-z|^{2}.=.2n+2tw.dot.y.0} with $w = -y_{0}$, we have 
    \begin{equation} \label{E:sum.si.sqrd}
        \sum_{i=1}^{n} s_{i}^{2} = (1/4) \sum_{i=1}^{n} |y_{i} + y_{0}|^{2} = (n-t)/2 .
    \end{equation}

Let $f(s) := \arcsin^{2} (s)$. We have,
    \begin{multline*}
        f'(s) = \frac{2 \arcsin s}{\sqrt{1 - s^{2}}}, \quad
          f''(s) = \frac{2}{1-s^{2}} + \frac{ 2 s \arcsin s }{(1 - s^{2})^{3/2}} \geq 0, \\
            f'''(s) = \frac{6 s}{(1-s^{2})^{2}} + \frac{2 \arcsin s}{(1-s^{2})^{3/2}} 
              + \frac{6 s \arcsin s}{(1-s^{2})^{5/2}}, \qquad 0 \leq s < 1.
    \end{multline*}
Thus, $f$ is convex on $[0,1]$ and $f(0)= 0$, $f'(0) = 0$, $f''(0) = 2$, and $f'''(0) = 0$. Since $f$ is $C^{\infty}$ we know by corollary \ref{C:cont.diff.=.loc.Lip}, that $f'''$ is locally Lipschitz. 
Hence, if $0 \leq u \leq s < 1$, then we have $f'''(u) = O(u)$ as $u \downarrow 0$. Therefore, expanding $\arcsin^{2} s_{i}$ about 0 in a Taylor expansion with remainder (Apostol \cite[Theorem 5--14, p.\ 96]{tmA57.Apostol}) 
we get for $s$ sufficiently close to 0, 
$f(s) = s^{2} + O(s) s^{3} = s^{3} = s^{2} \bigl[ 1 + O(s^{2}) \bigr]$. We thus get from \eqref{E:geod.dist.on.product.of.spheres},
    \begin{equation*}
        \rho(x,x')^{2} = 4 \sum_{i=1}^{n} s_{i}^{2} \bigl[ 1 + O(s_{i}^{2}) \bigr] 
          = \sum_{i=1}^{n} |y'_{i} - y_{i}|^{2}  \Bigl[ 1 + O \bigl( |y'_{i} - y_{i}|^{2} \bigr) \Bigr] .
    \end{equation*}
Hence, 
    \begin{equation}  \label{E:rho.above.below.bounds}
      \text{if } |x-x'| \text{ is sufficiently small, then }
      \tfrac{1}{2} |x-x'| < \rho(x,x') < \tfrac{3}{2} |x-x'| .
    \end{equation}

In particular, by \eqref{E:sum.si.sqrd}, for $|x - x_{0}|$ sufficiently small,
    \begin{equation*}
        \frac{4 \sum_{i=1}^{n} \arcsin^{2} s_{i} }{ 2(n-t) }
          = \frac{4 \sum_{i=1}^{n} \bigl( s_{i}^{2} +  O(s_{i}) s_{i}^{3} \bigr) }
            { 4 \sum_{i=1}^{n} s_{i}^{2} }
              = 1 + \frac{ \sum_{i=1}^{n} O(s_{i}) s_{i}^{3} }{ \sum_{i=1}^{n} s_{i}^{2} }.
    \end{equation*}

For $|x - x_{0}|$ sufficiently small there exists $K < \infty$ s.t.\ $O(s_{i})/s_{i} < K$ ($i \in \NN_{n}$) independent of $x$. Hence, by \eqref{E:sum.si.sqrd},
    \begin{multline} \label{E:geod.dist.to.y0.sqrd/2(n-t).=}
        \frac{4 \sum_{i=1}^{n} \arcsin^{2} s_{i} }{ 2(n-t) }
            = 1 + \frac{ \sum_{i=1}^{n} \tfrac{O(s_{i})}{s_{i}} s_{i}^{4} }{ \sum_{i=1}^{n} s_{i}^{2} }
              \leq 1 + K \frac{ \sum_{i=1}^{n} s_{i}^{4} }{ \sum_{i=1}^{n} s_{i}^{2} } \\
                \leq 1 + K \frac{ \left( \sum_{i=1}^{n} s_{i}^{2} \right)^{2} }
                  { \sum_{i=1}^{n} s_{i}^{2} }
                  = 1 + K \sum_{i=1}^{n} s_{i}^{2}
                    = 1 + (K/2) (n-t).
    \end{multline}

Let $x \in \Ss_{t}$ and for some $w \in S^{\nvar}$ let $x' := (w, \ldots, w) \in \T \subset \RR^{\nvar+1}$ be the closest point of $\T$ to $x$ w.r.t.\ the geodesic metric $\rho$. Such a $w$ exists since $\Ss_{t}$ is compact, by \eqref{E:mu.a.Sa.is.compact}. Thus, 
$\rho(x, x') = \text{dist} (x, \T)$, where $\text{dist} (x, \T)$ is the geodesic, i.e., $\rho$, distance from $x$ to $\T$. First, observe that on the punctured interval 
$(-1,0) \cup (0, 1)$ the function $| \arcsin |$ has derivative $\geq 1$. 
Thus, $|s| \leq | \arcsin s |$. ($s \in [-1, 1]$). Because of this, \eqref{E:bunch.of.neg.y0s.is.closest.to.T}, \eqref{E:geod.dist.on.product.of.spheres},  \eqref{E:geod.dist.to.y0.sqrd/2(n-t).=}, and the definition $s_{i} := |y_{i} + y_{0}|/2$,
    \begin{align}  \label{E:rho(x,x').sqrd.bounds}
        2(n-t) &= |x + x_{0}|^{2}  
        \leq |x - x'|^{2} = \sum_{i=1}^{n} |y_{i} - w|^{2} \notag \\
          &= 4 \sum_{i=1}^{n} \left(\frac{|y_{i} - w|}{2} \right)^{2} \notag \\
          &\leq 4 \sum_{i=1}^{n} \arcsin^{2} \frac{|y_{i} - w|}{2}  = \rho^{2}(x,x') \notag \\
          &= \text{dist} (x, \T)^{2}  \\
          &\leq \rho^{2}(x, x_{0}) = 4 \sum_{i=1}^{n} \arcsin^{2} \frac{|y_{i} + y_{0}|}{2} \notag \\
            &= 2(n-t)\bigl(1 + O(1)(n-t) \bigr) \notag . 
    \end{align}
Expanding about $u = 0$ we see $\sqrt{1+u} = 1 + u/2 + O(1) u^{2}$ as $u \downarrow 0$. Therefore, from \eqref{E:rho(x,x').sqrd.bounds}
    \begin{equation}  \label{E:geod.dist.to.y0.=.2(n-t).+}
        \text{dist} (x, \T) = \sqrt{2(n-t)} \sqrt{ 1 + O(1)(n-t) } 
           = \sqrt{2(n-t)} \bigl( 1 + O(1)(n-t) \bigr) .
    \end{equation}
This proves part \ref{I:aug.mean.geod.dist}, \eqref{E:rho.t.approx.sqrt.2(n-t)}, of the proposition.

\emph{Now we prove part \ref{I:rho.t.is.locally.Lip}.} \emph{Suppose} the following is true:
     \begin{multline}  \label{E:rho.x.x'(x).leq.Kbeta.|s-t|.|s-t|.leq.eta}      
          \text{For every } \beta \in (0, 1/2) \text{ there exists } K_{\beta} < \infty 
            \text{ and } \eta_{\beta} > 0 \\
            \text{ s.t.\ if } s, t \in (n-1 + \beta, n -\beta),  \quad |s-t| < \eta_{\beta}, 
              \text{ and } x \in \Ss_{t} \\
                \text{ then there exists } 
                  x'(x) \in \Ss_{s} \text{ s.t.\ } \rho \bigl( x, x'(x) \bigr) \leq K_{\beta} |s - t| .
    \end{multline}

Using the yet unproven \eqref{E:rho.x.x'(x).leq.Kbeta.|s-t|.|s-t|.leq.eta}, we prove part \ref{I:rho.t.is.locally.Lip} of the proposition. Let 
    \begin{equation}  \label{E:n-1<t<n}
      t \in (n-1, n) .
    \end{equation}
Also let 
    \begin{equation}  \label{E:beta.in.(0,1/2)}
      \beta \in \Bigl( 0, \min \bigl\{ t - (n-1), n-t \bigr\} \Bigr) \subset (0,1/2) .
    \end{equation}
Thus, \eqref{E:beta.in.(0,1/2)} implies $t \in (n-1 + \beta, n - \beta)$ as in \eqref{E:rho.x.x'(x).leq.Kbeta.|s-t|.|s-t|.leq.eta}.

Let $\eta_{\beta}$ and $K_{\beta}$ be as in \eqref{E:rho.x.x'(x).leq.Kbeta.|s-t|.|s-t|.leq.eta}. Since $\T$ and $\Ss_{t}$ are compact, by \eqref{E:mu.a.Sa.is.compact}, 
there exists $x_{t} \in \Ss_{t}$ s.t.\ $dist(x_{t}, \T) = \rho_{t}$, defined in \eqref{E:rho.t.approx.sqrt.2(n-t)}. Since $t \in (n-1 + \beta, n - \beta)$, the interval 
$J := (t-\eta_{\beta}, t+\eta_{\beta}) \cap (n-1 + \beta, n -\beta)$ is non-empty.
Let $s \in J$. Then $s,t \in (n-1 + \beta, n -\beta)$ and $|s-t| < \eta_{\beta}$. Hence, by \eqref{E:rho.x.x'(x).leq.Kbeta.|s-t|.|s-t|.leq.eta} there exists 
$x' = x'(x_{t}) \in \Ss_{s}$ s.t.\ 
$\rho(x_{t}, x') < K_{\beta} |s - t|$. Therefore, $dist(x', \T) < \rho_{t} + K_{\beta} |s - t|$ and, hence (see \eqref{E:set.distances}), $\rho_{s} < \rho_{t} + K_{\beta} |s - t|$. Reversing the roles of $s$ and $t$, we find
$\rho_{t} < \rho_{s} + K_{\beta} |s - t|$. Part \ref{I:rho.t.is.locally.Lip} of the proposition is now proved, modulo \eqref{E:rho.x.x'(x).leq.Kbeta.|s-t|.|s-t|.leq.eta}.

We make preparations for proving \eqref{E:rho.x.x'(x).leq.Kbeta.|s-t|.|s-t|.leq.eta}. WLOG, \eqref{E:y0.starts.with.0s}, holds: $y_{0} = (0, \ldots, 0, 1)$. 
Let $\beta \in (0,1/2)$; $t \in [n-1 + \beta, n -\beta]$, the closed interval; 
and $x \in \Ss_{t}$. As usual, write $x = (y_{1}, \ldots, y_{n})$. By \eqref{E:bunch.of.neg.y0s.is.closest.to.T}, 
$|x - (-x_{0})| = \sqrt{2(n-t)} < \sqrt{2}$ so by \eqref{E:y.in.terms.of.w}we may write
$y_{i} = (w_{i}^{1 \times \nvar}, -\sqrt{1 - | w_{i} |^{2} }^{1 \times 1}) \in S^{\nvar}$ 
($i \in \NN_{n}$). So
    \begin{equation}  \label{E:x.y.w.z.c}
      x \in \Ss_{t}, \; y_{i} = \bigl( w_{i}, -\sqrt{1 - | w_{i} |^{2} } \bigr) 
        \in S^{\nvar} ,
    \end{equation}
Write $c_{i} := |w_{i}|^{2}$.

Let 
    \begin{equation*}
      m := n-1 .
    \end{equation*}
By \eqref{E:Sa=Sa'} and \eqref{E:S'a.D'a.defns}, 
since $y_{0} = (0, \ldots, 0, 1)$, we have 
    \begin{equation} \label{E:sum.z=-t;sum.w=0}
      - \sum_{i=1}^{n} \sqrt{1 - c_{i}} = z_{1} + \cdots + z_{n} = - t < -m 
        \text{ and } w_{1} + \cdots + w_{n} = 0 .
    \end{equation}
Now, $z_{i} \geq -1$ for all $i \in \NN_{n}$ so, 
$-m + z_{i} \leq z_{1} + \cdots + z_{n} = - t$. I.e., 
    \begin{equation*}
      -1 \leq z_{i} \leq m-t < 0, \text{ for all } i \in \NN_{n} ,
    \end{equation*}
by \eqref{E:n-1<t<n}. It follows from \eqref{E:beta.in.(0,1/2)} 
that $- \sqrt{1 - c_{i}} = z_{i} \leq m-t \leq - \beta $. Therefore, 
    \begin{equation}  \label{E:uppr.bnd.on.ci}
      c_{i} \leq 1 - (t-m)^{2} \leq 1 - \beta^{2},  \text{ for all } i \in \NN_{n}. 
    \end{equation}
Since $-z_{1} - \cdots - z_{n} = t$, by \eqref{E:sum.z=-t;sum.w=0}, and 
$t \leq n-\beta$, there exists at least one $i$ 
s.t.\ $\sqrt{1 - c_{i}} \leq (n-\beta)/n$. Thus, by \eqref{E:uppr.bnd.on.ci},
    \begin{multline}  \label{E:ci.min.max.beta}    
      \text{If } t \in [n-1 + \beta, \, n -\beta], \text{ then } 
        1 - (1 - \beta/n)^{2} \leq \max_{i} c_{i} \leq 1 - \beta^{2} . \\
          \text{ The inequalities are strict if } t \in (n-1 + \beta, \, n -\beta). 
    \end{multline}

Let $r \geq 0$. Then, by \eqref{E:sum.z=-t;sum.w=0}, $r w_{1} + \cdots + r w_{n} = 0$. Suppose $r^{2} c_{i} \leq 1$ for all $i$. Let
    \begin{equation}  \label{E:x'[r].defn}
      x' := x'[r] := x'(r, x) :=  (y_{1}', \ldots, y_{n}') \text{ with }
        y_{i}' := \bigl( r w_{i}, z_{i}' \bigr), \text{ where } z_{i}' = - \sqrt{1 - r^{2} c_{i}}.
    \end{equation} 
Then $x' \in \Ss_{s}$, where $s = \sum_{i=1}^{n} \sqrt{1 - r^{2} c_{i}}$. 

Write $u := r^{2}$, so $u c_{i} \leq 1$ for all $i$. Suppose
$s := \sum_{i=1}^{n} \sqrt{1 - u c_{i}} \in [n-1+\beta, n- \beta]$. Then, by \eqref{E:ci.min.max.beta} with $s$ in place of $t$ and $u c_{i}$ in place of $c_{i}$, 
$n - \beta \geq s = \sum_{i=1}^{n} \sqrt{1 - u c_{i}} \geq n \sqrt{1 - u (1 - \beta^{2})}$. Now, $n - \beta \geq n \sqrt{1 - u (1 - \beta^{2})}$ is true if and only if 
    \begin{equation}  \label{E:lwr.bound.on.u}
      r^{2} = u \geq \frac{ n^{2} - (n - \beta)^{2} }{ n^{2} (1 - \beta^{2}) } .
    \end{equation} 
The RHS, $\bigl[ n^{2} - (n - \beta)^{2} \bigr] / \bigl[ n^{2} (1 - \beta^{2}) \bigr]$, of the preceding is strictly positive. Now, $n \geq 3$, by \eqref{E:n>2,nvar>0}, and $\beta < 1/2$ (by \eqref{E:beta.in.(0,1/2)}). It follows from a simple argument that the RHS of \eqref{E:lwr.bound.on.u} is strictly less than 1.

For $\mbf{v} = (v_{1}, \ldots, v_{n}) \in [0,1]^{n}$, define 
    \begin{equation}  \label{E:a(v).defn}
      a(\mbf{v}) := \sum_{i=1}^{n} \sqrt{1 - v_{i}} . 
    \end{equation} 
Then, if $x \in \Ss_{t}$ and $\mbf{v} = (c_{1}, \ldots, c_{n}) \in [0,1]^{n}$ 
where $c_{i} := |w_{i}|^{2}$, by \eqref{E:sum.z=-t;sum.w=0}, we have 
    \begin{equation}  \label{E:sum.sqrt(1-vi).=.t}
       a(\mbf{v}) = t.
    \end{equation}

For $\beta \in (0, 1/2)$, let 
    \begin{equation}  \label{E:V.beta.defn}
      V_{\beta} := a^{-1} \bigl( [n-1 + \beta, \, n -\beta] \bigr) \subset (0,1)^{n} .
    \end{equation} 
Thus, $V_{\beta}  \subset [0,1]^{n}$ is compact. Just as in \eqref{E:ci.min.max.beta} , we have 
    \begin{equation}  \label{E:max.vi.over.Vbeta}
      \mbf{v} = (v_{1}, \ldots, v_{n}) \in V_{\beta} \; \text{ implies }
         \max_{i} v_{i} \leq 1 - \beta^{2} .
    \end{equation}

If $\mbf{w} \in \overline{B_{1}^{\nvar}(0)}$ let  
$\blds{\langle w \rangle}^{2} := \bigl( |w_{1}|^{2}, \ldots, |w_{n}|^{2} \bigr)$.
For $\mbf{u} = (u_{1}, \ldots, u_{n}) \in [0, 1]^{n} 
\setminus \bigl\{ (0, \ldots, 0) \bigr\}$ define 
    \begin{equation*}
         J(\mbf{u}) := \bigl[ 0, (\max_{i} u_{i})^{-1/2} \bigr]
           \text{ and } J[\mbf{w}] := J \bigl( \blds{\langle w \rangle}^{2} \bigr) .
    \end{equation*}
\eqref{E:max.vi.over.Vbeta} implies that $[0,1/\sqrt{1-\beta^{2}} ] \subset J(\mbf{v})$ whenever $\mbf{v}) \in V_{\beta}$. Define
    \begin{multline}   \label{E:f(lambda,v).defn}
      f(\lambda) := f(\lambda, \mbf{v}) =  a(\lambda^{2} \mbf{v}) 
        = \sum_{i=1}^{n} \sqrt{1 - \lambda^{2} v_{i}} \leq n , \\
        \mbf{v} = (v_{1}, \ldots, v_{n}) \in [0, 1]^{n} \setminus \bigl\{ (0, \ldots, 0) \bigr\} , \;
          \lambda \in J(\mbf{v}) .
    \end{multline} 
$\max_{i} v_{i} \leq 1 - \beta^{2}$. 
Since $v \mapsto v/\sqrt{1 - v}$ is increasing in $v \in [0,1]$, we therefore have
    \begin{multline}  \label{E:f'(lambda)}
      \text{If } \mbf{v} \in V_{\beta} \text{ then }
        f'(\lambda) = \frac{\partial}{\partial \lambda} f ( \lambda, \mbf{v} )
          = - \sum_{i=1}^{n} \frac{\lambda v_{i}}{\sqrt{1 - \lambda^{2} v_{i}}} 
            \geq - n \frac{\lambda (1-\beta^{2})}{\sqrt{1 - \lambda^{2} (1-\beta^{2})}} , \\
              \text{ providing } 0 \leq \lambda < 1/\sqrt{1-\beta^{2}} .
    \end{multline}

Thus, given $\mbf{v}$, if $\mbf{v} = (v_{1}, \ldots, v_{n}) \in V_{\beta}$ then at least one $v_{i} \neq 0$. Hence, $f'(\lambda) < 0$ on its domain, providing $\lambda > 0$.
It follows that 
    \begin{equation}  \label{E:f.invrtbl}
      \lambda \mapsto f(\lambda, \mbf{v}) 
        \text{ is invertible on } \Bigl( f \bigl( 1/\sqrt{1-\beta^{2}} \bigr), n \Bigr].
    \end{equation} 

Suppose $t \in [n-1 + \beta, n -\beta]$. Then there exists $\mbf{v} \in V_{\beta}$ s.t.\ \eqref{E:sum.sqrt(1-vi).=.t} holds. Then
    \begin{equation}  \label{E:bounds.on.f}
      f \bigl[1/\sqrt{1 - \beta^{2}}, \mbf{v} \bigr] < f(1, \mbf{v} ) = a(\mbf{v}) = t
        \leq n - \beta < n = f(0).
    \end{equation} 
For $\mbf{v} \in V_{\beta}$, let 
$\eta(\mbf{v}) := \min \Bigl\{ a(\mbf{v}) - f \bigl[1/\sqrt{1 - \beta^{2}}, \mbf{v} \bigr] , 
n-a(\mbf{v}) \Bigr\}$, so $\eta(\mbf{v}) > 0$. If $s \in \RR$ and 
$|s - a(\mbf{v})| \leq \eta(\mbf{v})$ 
then $f \bigl[1/\sqrt{1 - \beta^{2}}, \mbf{v} \bigr] \leq s \leq n = f(0)$ so there exists 
$\lambda \in \bigl[ 0, 1/\sqrt{1 - \beta^{2}} \bigr]$ s.t.\ $f ( \lambda, \mbf{v} ) = s$. 

Since $\mbf{v} \mapsto \eta(\mbf{v})$ is strictly positive in $\mbf{v} \in V_{\beta}$, by \eqref{E:bounds.on.f}, and continuous on the compact set $V_{\beta}$, it achieves a finite minimum value, $\eta = \eta_{\beta} > 0$. Thus, 
    \begin{equation}   \label{E:when.f(lambda,v)=s}
      \text{ If } |s - a(\mbf{v})| < \eta_{\beta} \text{ there exists } 
        \lambda \in \bigl[ 0, 1/\sqrt{1 - \beta^{2}} \bigr] 
          \text{ s.t.\ } f ( \lambda, \mbf{v} ) = s .
    \end{equation}   
Moreover, since $f' < 0$ on its domain, it has an inverse. Therefore, 
    \begin{equation} \label{E:f.invrs.cont.}
      \text{As } s \to a(\mbf{v}) \text{ we have } u = f^{-1}(s) \to 1.
    \end{equation}

Finally we can prove \eqref{E:rho.x.x'(x).leq.Kbeta.|s-t|.|s-t|.leq.eta}. 
Let $\beta \in (0,1/2)$ and 
    \begin{equation}  \label{E:s,t.in.beta.n.intrvl}
      s, t \in (n-1+\beta, n- \beta) \text{ with } |s - t| \leq \eta .
    \end{equation}
Let $x \in \Ss_{t}$. Let $w_{i}^{1 \times \nvar}$ be as in \eqref{E:x.y.w.z.c}, write $c_{i} := |w_{i}|^{2}$ ($i \in \NN_{n}$), and let 
    \begin{equation*}
      \mbf{v} = (c_{1}, \ldots, c_{n}) \text{ so } a(\mbf{v}) = t.
    \end{equation*} 
Then, by \eqref{E:when.f(lambda,v)=s}, 
there exists $r \in \bigl[ 0, 1/\sqrt{1 - \beta^{2}} \bigr]$ 
s.t.\ $f(r, \mbf{v}) = s$. Let $x' := x'(x) := x'[r] := (y_{1}', \ldots, y_{n}')$, where
$y_{i}' = \bigl( r w_{i}^{1 \times \nvar}, - \sqrt{1 - r^{2} |w_{i}|^{2}} \bigr) \in S^{\nvar}$ 
($i \in \NN_{n}$). Thus,
    \begin{equation*}
      x' = x'(x) = x'[r] \in \Ss_{s} .
    \end{equation*}
(See \eqref{E:x'[r].defn}.) Write $u := r^{2}$. We have
    \begin{align}  \label{E:|x'-x|.bound}
      \bigl| x'[r] - x \bigr| &\leq \sum_{i=1}^{n} |y_{i}' - y_{i} | 
        \leq \sum_{i=1}^{n} |r w_{i} - w_{i} | 
          + \sum_{i=1}^{n} \bigl| \sqrt{1 - r^{2} |w_{i}|^{2}} - \sqrt{1 - |w_{i}|^{2}} \bigr| 
              \notag \\
            &= \sum_{i=1}^{n} |r - 1| \sqrt{c_{i}}
              + \sum_{i=1}^{n} \bigl| \sqrt{1 - r^{2} c_{i}} - \sqrt{1 - c_{i}} \bigr| \\
            &\leq n |r - 1|
              + \sum_{i=1}^{n} \bigl| \sqrt{1 - r^{2} c_{i}} - \sqrt{1 - c_{i}} \bigr| , \notag 
    \end{align}
since $c_{i} := |w_{i}|^{2}$. It follows from the preceding and \eqref{E:f.invrs.cont.} that, making $\eta$ smaller if necessary we may assume that \eqref{E:rho.above.below.bounds} holds for $x, x' = x'[r]$ 
if $|s-t| < \eta$.

Let $i \in \NN_{n}$. \emph{Claim:} 
    \begin{equation}  \label{E:r.ci.ineq}
      \bigl| \sqrt{1 - r^{2} c_{i}} - \sqrt{1 - c_{i}} \bigr| 
         \leq \frac{ \sqrt{1 - \beta^{2}}}{1 - \sqrt{1-\beta^{2}}} |1-r|,
             \quad \text{ if } r \in \bigl[ 0, 1/\sqrt{1 - \beta^{2}}\bigr].
    \end{equation}
We will use the following. By \eqref{E:ci.min.max.beta} (or \eqref{E:uppr.bnd.on.ci}), 
$0 \leq c_{i} \leq 1 - \beta^{2} < 1$. The function $u \mapsto u/(1-u)$ is increasing 
in $u \in [0, 1)$. Therefore, 
    \begin{equation}  \label{E:r.ci.ineq+}
      \frac{c_{i}}{\sqrt{1 - c_{i}}} = \frac{c_{i} \sqrt{1 - c_{i}}}{1 - c_{i}}
        \leq \frac{\sqrt{c_{i}} \sqrt{1-c_{i}}}{1 - \sqrt{c_{i}}}  
          \leq \frac{\sqrt{c_{i}} }{1 - \sqrt{c_{i}}} 
            \leq \frac{ \sqrt{1 - \beta^{2}}}{1 - \sqrt{1-\beta^{2}}} . 
    \end{equation}

To prove \eqref{E:r.ci.ineq}, we first, consider the case $r \leq 1$.
With $r \leq 1$ we have
$\bigl| \sqrt{1 - r^{2} c_{i}} - \sqrt{1 - c_{i}} \bigr| = \sqrt{1 - r^{2} c_{i}} - \sqrt{1 - c_{i}}$.
Hence, with $r \leq 1$, by \eqref{E:r.ci.ineq+}, it suffices to prove
    \begin{equation*}
       \sqrt{1 - r^{2} c_{i}} \leq \frac{c_{i}}{\sqrt{1 - c_{i}}} (1 - r) + \sqrt{1 - c_{i}} .
    \end{equation*}
This is true if and only if ($\Leftrightarrow$), 
    \begin{align*} 
      1 - r^{2} c_{i} - c_{i} + r^{2} c_{i}^{2} = (1 - c_{i})(1 - r^{2} c_{i}) 
        &\leq  \bigl[ c_{i} (1 - r) + (1 - c_{i}) \bigr]^{2} = r^{2} c_{i}^{2} - 2 r  c_{i} + 1 \\
      \Leftrightarrow - r^{2} c_{i} - c_{i} &\leq - 2 r  c_{i} \\
      \Leftrightarrow - r^{2} - 1 &\leq - 2 r \\
      \Leftrightarrow 0 &\leq r^{2} - 2 r + 1 = (r - 1)^{2} .
    \end{align*}
This proves \eqref{E:r.ci.ineq} in the case $r \leq 1$. 

Now suppose 
$r \in \bigl[ 1, 1/\sqrt{1 - \beta^{2}}\bigr] \subset \bigl[ 1, 1/\sqrt{c_{i}} \bigr]$, by \eqref{E:ci.min.max.beta}. 
Let 
    \begin{equation}   \label{E:big.C.in.terms.of.little.c}
      C := C_{\beta} := \frac{ \sqrt{1-\beta^{2}}}{ 1 - \sqrt{1-\beta^{2}}}
        \geq \frac{\sqrt{c_{i}} \sqrt{1-c_{i}}}{1 - \sqrt{c_{i}}}, 
    \end{equation}
by \eqref{E:r.ci.ineq+}. With $r \in \bigl[ 1, 1/\sqrt{c_{i}} \bigr]$ we have 
$\bigl| \sqrt{1 - r^{2} c_{i}} - \sqrt{1 - c_{i}} \bigr| = \sqrt{1 - c_{i}} - \sqrt{1 - r^{2} c_{i}}$ 
and $|1-r| = r-1$. Thus, \eqref{E:r.ci.ineq} in the case $r \in (1, 1/\sqrt{c_{i}}]$, is true if and only if 
    \begin{equation*}
      \sqrt{1 - c_{i}} \leq C (r-1) + \sqrt{1 - r^{2} c_{i}} .
    \end{equation*}
Define
    \begin{equation*}
      g(r) := g(r, c_{i}) := C (r-1) + \sqrt{1 - r^{2} c_{i}} - \sqrt{1 - c_{i}},
        \qquad r \in [1, 1/\sqrt{c_{i}}] .
    \end{equation*}
It suffices to show $g(r) \geq 0$ for $r \in [1, 1/\sqrt{c_{i}}]$. We determine the minimum value of $g$ on that interval. It is routine to show that $g$, as a function on $[1, 1/\sqrt{c_{i}})$, has non-positive second derivative. Thus, $g$ achieves its minimum in $[1, 1/\sqrt{c_{i}}]$ at one or more endpoints. Now, $g(1) = 0$ and it is easily seen that $g(1/\sqrt{c_{i}}) \geq 0$ if and only if $C \geq \sqrt{c_{i}} \sqrt{1 - c_{i}}/(1 - \sqrt{c_{i}})$. But by \eqref{E:big.C.in.terms.of.little.c} asserts just that.
Since $\bigl[ 1, 1/\sqrt{1 - \beta^{2}}\bigr] \subset \bigl[ 1, 1/\sqrt{c_{i}} \bigr]$, we have that $g$ is non-negative throughout the interval $\bigl[ 1, 1/\sqrt{1 - \beta^{2}}\bigr]$. This concludes the proof of claim \eqref{E:r.ci.ineq} in the case $r \in \bigl[ 1, 1/\sqrt{1 - \beta^{2}}\bigr]$, and, hence, in general. 

Recall the definition \eqref{E:x'[r].defn}. Just after \eqref{E:|x'-x|.bound}
we specified that $0 < \eta = \eta_{\beta}$ be sufficiently small that, if $|s-t| < \eta$ then \eqref{E:rho.above.below.bounds} holds. Therefore, by \eqref{E:rho.above.below.bounds}, \eqref{E:|x'-x|.bound}, and \eqref{E:r.ci.ineq}, there exists $D < \infty$ depending only on $n$ and $\beta$ (actually, only on $\beta$) s.t.\
    \begin{multline} \label{E:rho.dist.1-r.bound}
      \rho \bigl( x'[r], x \bigr) \leq 3 \bigl|  x'[r] - x \bigr|/2 \leq D |1-r |,  \\
        \quad s,t \in (n-1 + \beta, n -\beta), \; x \in \Ss_{t}, \; f(r, \mbf{v}) = s, \; 
          |s - t| < \eta .
    \end{multline}

Next, we bound $|1-r|$ by a multiple of $|s-t|$. Recall the definition, \eqref{E:f(lambda,v).defn}, of $f$. We have 
    \begin{equation*}
      f(r) = f( r, \mbf{v} ) 
        = \sum_{i=1}^{n} \sqrt{1 - r^{2} c_{i}} ,
           \quad r \in \bigl[ 0, 1/\sqrt{1-\beta^{2}} \bigr] , 
             \mbf{v} = (c_{1}, \ldots, c_{n}) \in \bigl[ 0, 1 - \beta^{2} \bigr]^{n} . 
    \end{equation*}
$f$ is invertible by \eqref{E:f.invrtbl}. Thus, $r = f^{-1}(s)$ and 
$1 = f^{-1}(t)$. By the Mean Value Theorem, there exists $\tilde{s}$ lying between $s$ and $t$ s.t.\ 
    \begin{equation}  \label{E:r.s-t.MVT.eqn}
      r = f^{-1}(s) = f^{-1}(t) + (f^{-1})'(\tilde{s})(s-t) 
        = 1 + (f^{-1})'(\tilde{s}) (s-t) .
    \end{equation}
Thus, to bound $|1-r|$ by a multiple of $|s-t|$, it suffices to bound 
$\bigl| (f^{-1})' \bigr|$ above. 
But $(f^{-1})' = 1/f'$. Thus, it suffices to bound $|f'|$ below, preferably by something depending only on $n$ and $\beta$. 

Now, by \eqref{E:ci.min.max.beta}, 
since $t \in [n-1 + \beta, n -\beta]$, there exists $j \in \NN_{n}$ s.t.\ 
$c_{j} \geq 1 - (1 - \beta/n)^{2}$. Thus, by \eqref{E:f'(lambda)} and \eqref{E:lwr.bound.on.u}, 
    \begin{equation*}
      \left| \frac{\partial}{\partial u} f(u, \mbf{v}) \right|
        = \sum_{i=1}^{n} \frac{u c_{i}}{\sqrt{1 - u c_{i}}} \geq u \sum_{i=1}^{n} c_{i}
          \geq \frac{ \sqrt{ n^{2} - (n - \beta)^{2} } }{ n \sqrt{1 - \beta^{2}} } 
              \bigl[ 1 - (1 - \beta/n)^{2} \bigr] .
    \end{equation*}
Thus, by \eqref{E:r.s-t.MVT.eqn}, 
    \begin{equation}  \label{E:1-r.s-t.bound}
     |1-r| \leq \frac{n \sqrt{1 - \beta^{2}}}
         { \sqrt{ n^{2} - (n - \beta)^{2} } \, \bigl[ 1 - (1 - \beta/n)^{2} \bigr]}  \, |s-t| .
    \end{equation}

Let $D < \infty$ be as in \eqref{E:rho.dist.1-r.bound}. 
Then with 
    \begin{equation*}
      K_{\beta} := \frac{D \, n \sqrt{1 - \beta^{2}}}
             { \sqrt{ n^{2} - (n - \beta)^{2} } \, \bigl[ 1 - (1 - \beta/n)^{2} \bigr]}
    \end{equation*}  
\eqref{E:rho.x.x'(x).leq.Kbeta.|s-t|.|s-t|.leq.eta}  
follows from \eqref{E:rho.dist.1-r.bound}, and \eqref{E:1-r.s-t.bound}. Notice that as $\beta \downarrow 0$, $K_{\beta} \uparrow \infty$. Thus, we may only claim a local Lipschitz property. As observed above, part \ref{I:rho.t.is.locally.Lip} of the proposition follows.

\emph{Next, we prove part \ref{I:aug.mean.essential.dist.approx} of the proposition.} 
Let $t \in (n-1,n)$ be arbitrary but fixed. In particular, $t$ is not an integer. We already know, from \eqref{E:mu.a.Sa.is.compact}, \eqref{E:Hm.S.less.U.0}, and \eqref{E:dim.Sa.tilde} that 
    \begin{equation}  \label{E:St.is.smooth.manf}
      \Ss_{t} \text{ is a compact smooth } (n \nvar - \nvar - 1 )\text{-manifold.} 
    \end{equation} 

Suppose $dist_{n\nvar-\nvar-1}(\Ss_{t}, \T) > dist(\Ss_{t}, \T)$ and  
let $r \in \bigl( dist(\Ss_{t}, \T) , \, dist_{n\nvar-\nvar-1}(\Ss_{t}, \T) \bigr)$. (See \eqref{E:set.distances}.)
Let $\mcl{A} := \bigl\{ y \in \Ss_{t} : dist(y, \T) < r \bigr\}$, so $\mcl{A} \neq \varnothing$. By example \ref{Ex:dist.is.Lip}, $\mcl{A}$ is open in $\Ss_{t}$. Let $x \in \mcl{A}$. By lemma \ref{L:coord.maps.are.Lip} part \ref{I:phi.psi.Lip} there exists a coordinate neighborhood, 
$\clU \subset \mcl{A}$, of $x$ in $\Ss_{t}$ with Lipschitz coordinate map 
$\varphi : \clU \to \RR^{n\nvar-\nvar-1}$. Therefore, by \eqref{E:when.Haus.meas.=.Leb.meas} and \eqref{E:Lip.magnification.of.Hm}, 
$\Hm^{n \nvar - \nvar - 1}(\mcl{A}) > 0$. 
But this contradicts $r < dist_{n\nvar-\nvar-1}(\Ss_{t}, \T)$. 
Hence, $R_{t} = dist_{n\nvar-\nvar-1}(\Ss_{t}, \T) = dist(\Ss_{t}, \T) = \rho_{t}$. Part \ref{I:aug.mean.essential.dist.approx} of the proposition now follows from parts \ref{I:aug.mean.geod.dist} and \ref{I:rho.t.is.locally.Lip}. 

\emph{Now we work on part \ref{I:aug.mean.R.vol.bounded}.} 
Supppose \eqref{E:a.tween.n-1.n} holds: $a \in (n-1, n)$. By \eqref{E:St.is.smooth.manf}, $\Ss_{a}$ is a compact smooth 
$(n \nvar - \nvar - 1 )$-manifold. By Boothby \cite[Theorem (5.7) and Definitions (5.3) and (5.1), pp.\ 79, 77, 75]{wmB75}, $\Ss_{a}$ is a finite union of smooth images of closed $(n \nvar - \nvar - 1)$-dimensional cubes. Therefore, by corollary \ref{C:cont.diff.=.loc.Lip} and \eqref{E:Lip.magnification.of.Hm}, we have
	\begin{equation} \label{E:Sa.has.finite.vol}
		\Hm^{n \nvar - \nvar - 1}( \Ss_{a}) < \infty.
	\end{equation}

Let $x_{0} := (y_{0}, \ldots, y_{0}) \in \T$, where $y_{0} = (0, \ldots, 0, 1)$ as in \eqref{E:y0.starts.with.0s}. We define a map on a neighborhood of $-x_{0}$ in $\D$ which maps $\Ss_{a}$ onto $\Ss_{t}$ for any $t \in [a, n)$ and we identify a Lipschitz constant for the the map. Let $x = (y_{1}, \ldots, y_{n}) \in \D$. Now, regardless whether or not $x$ is a singularity, by \eqref{E:y.in.terms.of.w}, if it is sufficiently close to $-x_{0}$, then there exist let $w_{i} \in \overline{B_{1}^{\nvar}(0)}$ ($i \in \NN_{n}$) s.t.\ 
	\begin{equation}  \label{E:x.w.in.B}
		x = \bigl( \ldots, (w_{i}, -\sqrt{1 - | w_{i} |^{2} }), 
		  \ldots \bigr)^{1 \times n(\nvar +1)}.
	\end{equation} 
Conversely, for any $w_{1}, \ldots, w_{n} \in \overline{B_{1}^{\nvar}(0)}$ the point $x$ defined by \eqref{E:x.w.in.B} is in $\D$. Thus, \eqref{E:x.w.in.B} defines a coordinate neighborhood of $-x_{0}$.

Since \eqref{E:y0.starts.with.0s} and \eqref{E:a.tween.n-1.n} hold, by \eqref{E:S'a.D'a.defns} and \eqref{E:Sa=Sa'},
	\begin{equation} \label{E:St.wi.condn}
		x \in \Ss_{a} \text{ if and only if } \sum_{i=1}^{n} w_{i} = 0 \in \RR^{\nvar}; 
		        \text{ and } \sum_{i=1}^{n} \sqrt{1 - | w_{i} |^{2} }) = a  \in (n-1, n).
	\end{equation}

We define a superset of $\Ss_{a}$:
    \begin{equation}   \label{E:S.tilde.defn}
      \tilde{\Ss}_{a} := \bigl\{ x \in \D : 
        x \text{ has the form \eqref{E:x.w.in.B} and } 
          \sqrt{1 - | w_{i} |^{2} }) = a \bigr\} .
    \end{equation}

Write $u_{i} := | w_{i} |^{2}$, $i \in \NN_{n}$. 
\emph{Some} $i \in \NN_{n}$, must satisfy $\sqrt{1-u_{i}} \leq a/n$. 
Thus, as in \eqref{E:ci.min.max.beta}, 
    \begin{equation}  \label{E:ui.min.max.t}
        \text{If } x \in \tilde{\Ss}_{a} \text{ then }
          1 - \left( \frac{a}{n} \right)^{2} \leq \max_{i} u_{i} \leq 1 - \bigl[ a - (n-1) \bigr]^{2} .
    \end{equation}

For $x \in \D$ as in \eqref{E:x.w.in.B} 
and $\lambda \in \bigl[ -1/\max_{i} |w_{i}|, \, 1/\max_{i} |w_{i}| \bigr]$ , define
	\begin{equation}   \label{E:xi.x.lambda.defn}
	    \xi(x, \lambda) := \Bigl( \ldots, 
		\bigl( \lambda w_{i}, - \sqrt{1 - \lambda^{2} | w_{i} |^{2} } \bigr), 
		     \ldots \Bigr) \in (S^{\nvar})^{n}.
	\end{equation}
So $\xi(x, \lambda) \in \D$, $\xi(x, 1) = x$, 
and $\xi(x, 0) = -(y_{0}, \ldots, y_{0}) = - x_{0}$. Recall the definition, \eqref{E:f(lambda,v).defn}, of $f$. If $x \in \D$ satisfies \eqref{E:y.in.terms.of.w}, we have 
$\xi(x, \lambda) \in \tilde{\Ss}_{f(\lambda, \blds{\langle w \rangle}^{2})}$, where 
$\blds{\langle w \rangle}^{2} := \bigl( |w_{1}|^{2}, \ldots, |w_{n}|^{2} \bigr)$, 
for $\lambda \in J[\mbf{w}]$.

Conversely, we \emph{claim:} 
    \begin{equation}  \label{E:xi(x)=x'.for.some.x.lambda}
     \text{For every } t \in [a,n) \text{ and } x' \in \tilde{\Ss}_{t} 
       \text{ there exists } x \in \tilde{\Ss}_{a} \text{ and } \lambda \in J_{a} 
         \text{ s.t.\  } \xi(x, \lambda) = x' .
    \end{equation}
If $t = a$ take $\lambda = 1$. Let $t \in (a, n)$ and let $x' \in \tilde{\Ss}_{t}$. Write 
    \begin{equation}  \label{E:x'.and.w'}
      x' = \Bigl( \ldots, \bigl(w'_{i}, -\sqrt{1 - | w'_{i} |^{2} } \bigr), 
        \ldots \Bigr)^{1 \times n(\nvar +1)},
          \qquad w'_{1}, \ldots, w'_{n} \in \overline{B_{1}^{\nvar}(0)}.
    \end{equation} 
(See \eqref{E:Euc.ball.defn}.) It suffices to show that there exists $\lambda \in [ 0, 1)$  
and $w_{1}, \ldots, w_{n} \in \overline{B_{1}^{\nvar}(0)}$ s.t.\
	\begin{equation}
		\lambda w_{i} = w'_{i} \; \;  (i \in \NN_{n}) \text{ and }
	          \sum_{i=1}^{n} \sqrt{1 - |w_{i}|^{2} } = a.
	\end{equation}
Let $\lambda_{t} := \max_{i \in \NN_{n}} |w'_{i}|\in [0, 1]$. Then, by \eqref{E:w.bdd.away.from.ball.bndry} and \eqref{E:minmax.tilde.w}, with $w'$'s instead of $w$'s, 
    \begin{equation*}
      \sqrt{1 - t^{2}/n^{2}} \leq \lambda_{t} \leq \min \bigl\{\sqrt{2(n-t)} , 1 \bigr\} ,
    \end{equation*}
so $\lambda_{t} \to 0$ as $t \uparrow n$. 

Let 
    \begin{equation*}
      g(\lambda) := \sum_{i=1}^{n} \sqrt{1 - \lambda^{-2} |w'_{i}|^{2} } 
        = f \Bigl[ 1/\lambda ,  \blds{\langle w' \rangle}^{2} \bigr] ,
          \qquad 1/\lambda \in J[\mbf{w'}] .
    \end{equation*}
(See \eqref{E:f(lambda,v).defn}.) Thus, $g(\lambda) \in [0, n)$ 
for $1/\lambda \in J[\mbf{w'}]$ and $g(1) = t \geq a$. But for at least one $i$ we have 
$\lambda_{t}^{-2} |w'_{i}|^{2} = 1$. 
Therefore, $g(\lambda_{t})$ is the sum of $n-1$ numbers all between 0 and 1. 
Thus, since $a \in (n-1,n)$, 
    \begin{equation*}
      g(\lambda_{t}) \leq n-1 < a < t = g(1) .
    \end{equation*}
By \eqref{E:f'(lambda)}, $(\partial /\partial \lambda) f(\lambda, \mbf{v}) < 0$. Therefore, 
$g'(\lambda) > 0$ for $\lambda \geq \lambda_{t}$. It follows that there is 
a unique $\lambda(x') \geq \lambda_{t}$ s.t.\ $g(\lambda(x')) = a$. Recall the definition \eqref{E:S.tilde.defn} of $\tilde{\Ss}_{a}$. Then 
    \begin{multline}  \label{E:x(x').defn}
        \text{if }x(x') := \Bigl( \ldots, \lambda(x')^{-1} w'_{i}, 
          - \sqrt{1 - \lambda(x')^{-2} | w'_{i} |^{2} }, \ldots \Bigr) \in \tilde{\Ss}_{a} \\
            \text{ then } \xi \bigl[ x(x'), \lambda(x') \bigr] = x'.
    \end{multline} 
This completes the proof of the claim \eqref{E:xi(x)=x'.for.some.x.lambda} that there exists $x \in \tilde{\Ss}_{a}$ and $\lambda \in [0, 1)$ s.t.\ $\xi(x, \lambda) = x'$.

Let $\beta \in (0, 1/2)$. Recall \eqref{E:a(v).defn} and \eqref{E:V.beta.defn}. Suppose 
    \begin{equation} \label{E:u.a(u).beta.rule}
      \mbf{u} \in (0,1)^{n} \in V_{\beta}. \text{ I.e., }
        a(\mbf{u}) \in (n-1 + \beta, n - \beta) \neq \varnothing.
    \end{equation}
Eventually, we will fix $a = a(\mbf{u}) \in (n-1,n)$ and take 
$\beta \in \Bigl( 0, \min \bigl\{ a - (n-1), n-a \bigr\} \Bigr) \subset (0,1/2)$ but for now we grant $\mbf{u}$ more freedom.

Recall the definition, \eqref{E:f(lambda,v).defn}, of $f$. By \eqref{E:f'(lambda)},  
$f$ is differentiable and
    \begin{equation}  \label{E:partial.tau.lambda}
        \frac{\partial}{\partial \lambda} f(\lambda, \mbf{u})
          =  - \sum_{i=1}^{n} \frac{\lambda u_{i}}{\sqrt{1 - \lambda^{2} u_{i}}} 
            \leq -\lambda^{-1} \sum_{i=1}^{n} \lambda^{2} u_{i} < 0 ,
              \qquad \lambda \in (0,1) .
    \end{equation}

Let 
    \begin{equation} \label{E:t.in(a(u),n)}
      t \in \bigl( a(\mbf{u}), n \bigr) .
    \end{equation}  
Since $f(1,\mbf{u}) = a(\mbf{u})$, $f(0,\mbf{u}) = n$, 
and $f(\lambda,\mbf{u})$ is strictly decreasing in $\lambda \in [0, 1)$, 
we have that there exists a unique 
    \begin{equation}  \label{E:lambda.t.u.solves.eqn}
      \lambda(t,\mbf{u}) \in (0,1) \text{ and }
      f \bigl( \lambda(t,\mbf{u}),\mbf{u} \bigr) = t  \text{ so, by \eqref{E:S.tilde.defn}, }
        \xi \bigl( x, \lambda(t,\mbf{u}) \bigr) \in \tilde{\Ss}_{t} .
    \end{equation}
Moreover, by \eqref{E:minmax.tilde.w} for some $i$ we have $\lambda(t,\mbf{u})^{2} u_{i}  \geq 1 - t^2/n^{2}$. Therefore, by \eqref{E:partial.tau.lambda} 
\eqref{E:minmax.tilde.w}, we have 
    \begin{equation} \label{E:partial.tau.lambda(t,u).bound}
        \frac{\partial}{\partial \vartheta} f(\vartheta, \mbf{u}) 
          \restriction_{\vartheta = \lambda(t,\mbf{u}) } 
            \leq  -\frac{ 1  - t^{2}/n^{2} }{\lambda(t,\mbf{u}) } < 0 .
    \end{equation}

Let $x$ be as in \eqref{E:y.in.terms.of.w} and let 
$\blds{\langle w \rangle}^{2} = \bigl( |w_{1}|^{2}, \ldots, |w_{n}|^{2} \bigr)$. We analyze the map 
$(\mbf{w},t) \mapsto \xi \bigl[ x, \lambda(t, \blds{\langle w \rangle}^{2}) \bigr] \in \D$ 
($t \in \bigl( a(\blds{\langle w \rangle}^{2}), n \bigr)$). 
Now, $\lambda(t,\mbf{u}) \in (0,1)$ so by \eqref{E:partial.tau.lambda(t,u).bound},  
\eqref{E:lambda.t.u.solves.eqn} and the Implicit Function Theorem 
(Rudin \cite[Theorem 9.18, p.\ 196]{wR64.PMA}), 
$\lambda(t, \mbf{u})$ is differentiable in $\mbf{u}$. Moreover, 
by \eqref{E:1-|w|2.>.delta}, 
    \begin{equation*} \label{E:partial.tau.lambda.u.v}
        \frac{\partial}{\partial v_{i}} f \bigl[ \lambda(t,\mbf{u}), \mbf{v} \bigr] 
                = \frac{ \lambda(t,\mbf{u})^{2} }
                  {2 \sqrt{1 - \lambda(t,\mbf{u})^{2} v_{i}}} \\ 
                  \geq \frac{ \lambda(t,\mbf{u})^{2} }{2 \sqrt{1 - v_{i}}}
                    \geq  - \frac{ \lambda(t,\mbf{u})^{2} }{2 \delta_{a(\mbf{v})}} .
    \end{equation*}
In particular, since $\lambda(t,\mbf{u}) \in (0,1)$, we have, 
    \begin{equation}  \label{E:partial.tau.lambda.u.v.bnd}
     0 < \left| \frac{\partial}{\partial v_{i}} f \bigl[ \lambda(t,\mbf{u}), \mbf{v} \bigr] \right| 
         \leq \frac{ \lambda(t,\mbf{u})^{2} }{2 \delta_{a(\mbf{v})}} .
    \end{equation}

Let $i \in \NN_{n}$. By \eqref{E:lambda.t.u.solves.eqn}, 
$f \bigl[ \lambda(t,\mbf{u}), \mbf{u} \bigr] = t$ is constant in $\mbf{u}$. Therefore, applying multivariate chain rule to \eqref{E:lambda.t.u.solves.eqn} we get,
    \begin{equation*}
      0 = \frac{\partial}{\partial u_{i}} f \bigl[ \lambda(t,\mbf{u}), \mbf{u} \bigr] 
        = \frac{\partial}{\partial \vartheta} f(\vartheta, \mbf{u}) 
          \restriction_{\vartheta=\lambda(t,\mbf{u})} \,
            \frac{\partial}{\partial u_{i}}\lambda(t,\mbf{u}) 
              + \frac{\partial}{\partial u_{i}} f \bigl[\lambda(t,\mbf{u}), \mbf{u} \bigr] .
    \end{equation*}
Therefore, by \eqref{E:partial.tau.lambda.u.v.bnd}
and \eqref{E:partial.tau.lambda(t,u).bound}, 
     \begin{multline} \label{E:partial.lambda.partial.u.bdd}
        \left| \frac{\partial}{\partial u_{i}}\lambda(t, \mbf{u}) \right| 
          = \left| \frac{ \frac{\partial}{\partial u_{i}} f \bigl(\lambda(t, \mbf{u}), \mbf{u} \bigr) }
                 { \frac{\partial}{\partial \vartheta} f (\vartheta, \mbf{u})
                  \restriction_{\vartheta=\lambda(t,\mbf{u})} } \right| \\
            \leq \frac{ \lambda(t, \mbf{u})^{3} }
              {2 \delta_{a(\mbf{u})} (1 - t^{2}/n^{2})} 
                =  \frac{ n^{2} \lambda(t, \mbf{u})^{3} }
                  { \delta_{a(\mbf{u})} (n+t) \cdot 2(n-t)} < \infty.
    \end{multline} 

By \eqref{E:minmax.tilde.w}, for some $i \in \NN_{n}$, we have 
$| w_{i} | \geq \sqrt{1 - a(\mbf{u})^{2}/n^{2}}$. Therefore, by \eqref{E:1-|w|2.>.delta} and \eqref{E:lambda.t.u.solves.eqn}, there exists $i \in \NN_{n}$ s.t.\ 
    	\begin{equation*}
    		1 - \delta_{t}^{2} \geq \lambda(t, \mbf{u})^{2} |w_{i}|^{2} 
    		  \geq \lambda(t, \mbf{u})^{2} \bigl( 1 - a(\mbf{u})^{2}/n^{2} \bigr).
    	\end{equation*}
By \eqref{E:t.in(a(u),n)}, $n-1 < a(\mbf{u}) \leq t < n$, so by \eqref{E:delta.a.defn}, for this particular $i$ we have,
    	\begin{align*}
    		\lambda(t, \mbf{u})^{2} \bigl( 1 - a(\mbf{u})^{2}/n^{2} \bigr) 
    			 &\leq 1 - (t - n + 1)^{2} \\
    			 &= -(n-t)^{2} + 2(n-t) 
    			   =  (n-t) \bigl[ 2 - (n-t) \bigr]  \\ 
    			 &< 2(n-t).
    	\end{align*}
Thus,
    	\begin{equation}  \label{E:lambda.t.u.ineq}
    		\lambda(t, \mbf{u}) \leq \bigl( 1 - a(\mbf{u})^{2}/n^{2} \bigr)^{-1/2} \sqrt{2(n-t)} 
		  = K \bigl[ a(\mbf{u}) \bigr] \sqrt{2(n-t)} .
    	\end{equation}
So $K = K \bigl[ a(\mbf{u}) \bigr] < \infty$ does not depend on $t$ and only on $\mbf{u}$ through $a := a(\mbf{u})$.

Combining this with \eqref{E:partial.lambda.partial.u.bdd}, we get
     \begin{equation} \label{E:partial.lambda.partial.u.sqrt(n-t).bound}
        \left| \frac{\partial}{\partial v_{i}} \lambda(t, \mbf{v}) 
          \restriction_{\mbf{v}=\mbf{u}} \right| 
            \leq \frac{ n^{2} K^{3} \sqrt{2(n-t)} }{\delta_{a} (n+t)} 
              \leq \frac{ n K^{3} \sqrt{2(n-t)} }{\delta_{a} }
                \leq \frac{ \sqrt{2} n K^{3} }{\delta_{a} } .
    \end{equation}

Suppose $x \in \D$ is given by \eqref{E:x.w.in.B}, where  
$\mbf{w} := (w_{1}, \ldots, w_{n}) \in \bigl( \overline{B_{1}^{\nvar}(0)} \bigr)^{n}$. Suppose 
$\blds{\langle w \rangle}^{2} := \bigl( |w_{1}|^{2}, \ldots, |w_{n}|^{2} \bigr) = \mbf{u}$.
Thus, $x \in \tilde{\Ss}_{a(\mbf{u})}$. 
Recall the definition, \eqref{E:xi.x.lambda.defn}, of $\xi$. Define 
    \begin{equation}  \label{E:Omega(w).defn}
      \Omega_{t}(\mbf{w}) 
        := \bigl( \Omega_{t1}(\mbf{w}), \ldots, \Omega_{tn}(\mbf{w}) \bigr)
          := \xi \bigl( x, \lambda(t, \mbf{u}) \bigr) \in (S^{\nvar})^{n} .
    \end{equation} 
Taking $x' := \Omega_{t}(\mbf{w})$, express $x$ as in \eqref{E:x'.and.w'}. Thus, 
    \begin{equation}   \label{E:w'=lambda.w}
      w'_{i} =  \lambda(t, \mbf{u}) w_{i}, \quad i \in \NN_{n} .
    \end{equation}
By \eqref{E:lambda.t.u.solves.eqn} and \eqref{E:f(lambda,v).defn}, 
$a(\blds{\langle w' \rangle}^{2}) = a \bigl[ \lambda(t, \mbf{u})\mbf{u} \bigr] = t$. Hence, $\Omega_{t}(\mbf{w}) \in \tilde{\Ss}_{t}$.

By \eqref{E:xi.x.lambda.defn}, 
$\Omega_{t}(\mbf{w})$ has $n$ components, each in $S^{\nvar} \subset \RR^{\nvar+1}$. 
Write $w_{i} = \bigl( w_{i1}, \ldots, w_{i(\nvar+1)} \bigr) \in \overline{B_{1}^{\nvar}(0)}$. 
Let $i \in \NN_{n}$ and $j = 1, \ldots, \nvar$. Then the derivative 
of $\mbf{u} :=  \blds{\langle w \rangle}^{2} \in \RR^{n}$ w.r.t.\ $w_{ij}$ is all 0, except in the $i^{th}$ component where one finds $2 w_{ij}$ in the $j^{th}$ coordinate. By \eqref{E:xi.x.lambda.defn}, 
in $\tfrac{\partial}{\partial w_{ij}} \Omega_{t}(\mbf{w})$ only the $i^{th}$ component, 
$\Omega_{ti}(\mbf{w})$, can be nonzero. 
Write 
$\tfrac{\partial}{\partial h_{ij}} \Omega_{ti}(\mbf{h}) \restriction_{\mbf{h}=\mbf{w}} = (C^{j}_{i}, S^{j}_{i}) \in \RR^{\nvar+1}$, where $C^{j}_{i} \in \RR^{\nvar}$ and $S^{j}_{i} \in \RR$. 
Let $e_{j} = (0, \ldots, 0, 1, 0, \ldots, 0) \in \RR^{\nvar}$, 
where the ``1'' is in the $j^{th}$ position ($j = 1, \ldots, \nvar$). We have
    \begin{equation*}
      C^{j}_{i} = \left( \frac{\partial}{\partial v_{i}} \lambda(t, \mbf{v}) 
          \restriction_{\mbf{v}=\mbf{u}} \times 2 w_{ij} \right) w_{i}
            + \lambda(t, \mbf{\mbf{u}}) e_{j} .
    \end{equation*}
Therefore, by \eqref{E:lambda.t.u.ineq} and 
\eqref{E:partial.lambda.partial.u.sqrt(n-t).bound},
    \begin{equation}   \label{E:Si.bound}
        | C^{j}_{i} | \leq K' \sqrt{2(n-t)},
    \end{equation}
where $K' < \infty$ does not depend on $t$ or $\mbf{w}$. 
(But it does depend on $a(\mbf{u})$.)

Next, $S^{j}_{i}$:
    \begin{multline} \label{E:Si.expression}
        S^{j}_{i}  =  - \frac{1}{2} 
          \Bigl( 1 - \lambda(t, \mbf{u})^{2} | w_{i} |^{2} \Bigr)^{-1/2} \\
            \left( - \lambda(t, \mbf{u})^{2} \times 2 w_{ij} 
              - 2 \lambda(t, \mbf{u})  
                \frac{\partial}{\partial v_{i}} \lambda(t, \mbf{v}) \restriction_{\mbf{v}=\mbf{u}} 
                  \times |w_{ij}|^{2} \right) .
    \end{multline}
By \eqref{E:1-|w|2.>.delta} and \eqref{E:delta.a.defn}
(with $w'_{i} =  \lambda(t, \mbf{u}) w_{i}$ in place of $w_{i}$ and $t$ in place of $a$) and \eqref{E:t.in(a(u),n)},
    \begin{equation} \label{E:lambda^2.wi2.bound}
        \Bigl( 1 - \lambda(t, \mbf{u})^{2} | w_{i} |^{2} \Bigr)^{-1/2} 
          \leq 1 / \delta_{t}  \leq 1 / \delta_{a} .
    \end{equation}
Similarly, by \eqref{E:w.bdd.away.from.ball.bndry} and \eqref{E:lambda.t.u.solves.eqn}, we have 
    \begin{equation} \label{E:partial.lambda.partial.u.etc.bound}
        \lambda(t, \mbf{u})^{2} \times 2 |w_{ij}| 
          = 2 \lambda(t, \mbf{u}) |w'_{i}|
            \leq 2 \lambda(t, \mbf{u}) \sqrt{2(n - t)} \leq 2 \sqrt{2(n - t)} .
    \end{equation}
Similarly, by \eqref{E:1-|w|2.>.delta} again, 
    \begin{equation*}
         \lambda(t, \mbf{u})  |w_{i}|^{2} 
           \leq \lambda(t, \mbf{u})  |w_{i}|  \leq \sqrt{2(n-t)} . 
    \end{equation*}

Combining \eqref{E:Si.expression}, \eqref{E:lambda^2.wi2.bound}, \eqref{E:partial.lambda.partial.u.etc.bound}, 
and \eqref{E:partial.lambda.partial.u.sqrt(n-t).bound} we find that 
there exists $K'' < \infty$ not depending on $t$ and only depending on $\mbf{w}$ through $a \bigl( \blds{\langle w \rangle}^{2} \bigr) = a(\mbf{u})$ s.t.\
    \begin{equation*}
        |S^{j}_{i}| \leq K'' \sqrt{2(n-t)}.
    \end{equation*}
Combining this with \eqref{E:Si.bound}, we get finally that there exists $K''' < \infty$ s.t.\
    \begin{equation}  \label{E:partial.Xi.bound}
         \left| \frac{\partial}{\partial h_{ij}} \Omega_{t}(\mbf{h}) 
           \restriction_{\mbf{h} = \mbf{w}} \right| = \bigl| (C^{j}_{i}, S^{j}_{i}) \bigr| 
             \leq K''' \sqrt{2(n-t)} ,
    \end{equation}
where $K''' < \infty$ does not depend on $t$ and only depending on $\mbf{w}$ through $a \bigl( \blds{\langle w \rangle}^{2} \bigr) = a(\mbf{u})$.

Let $j \in \NN_{n}$; $\mu = 1, \ldots, \nvar$; and $\epsilon = \pm 1$  be arbitrary and let $W_{j,\mu,\epsilon} = W_{j,\mu,\epsilon}(\mbf{u})$ be the set of all 
$(w_{1}, \ldots, w_{n}) \in \bigl( \overline{B_{1}^{\nvar}(0)} \bigr)^{n}$ (as usual the bar indicates closure) s.t.\  
    \begin{equation*}
      |w_{i}|^{2} < 1 - \bigl[ a - (n-1) \bigr]^{2} = 1-\delta_{a}^{2} \; \;
        (i \in \NN_{n}) \text{ and }
          \epsilon w_{j,\mu} >  \sqrt{ \nvar^{-1} \,  
            \left[ 1 - \left( \tfrac{a}{n} \right)^{2} \right] } .
    \end{equation*}
Here, $a = a(\mbf{u})$ and $w_{j} = (w_{j,1}, \ldots, w_{j,\nvar})$. 
Then $W_{j,\mu,\epsilon} \subset \RR^{n \nvar}$ is open and convex. Moreover, by \eqref{E:ui.min.max.t} with $a$ in place of $t$, 
if $x = (y_{1}, \ldots, y_{n}) \in \tilde{\Ss}_{a}$ satisfies \eqref{E:y.in.terms.of.w} 
then for some $j \in \NN_{n}$; $\mu = 1, \ldots, \nvar$; and $\epsilon = \pm 1$ we have 
$(w_{1}, \ldots, w_{n}) \in \overline{W}_{j,\mu,\epsilon}$.

By \eqref{E:lambda.t.u.solves.eqn}, $\lambda(t,\mbf{u}) \in (0,1)$ (see also \eqref{E:t.in(a(u),n)}) so the domain of $\Omega_{t}$ certainly includes $W_{j,\mu,\epsilon}$. (See \eqref{E:Omega(w).defn} and \eqref{E:xi.x.lambda.defn}.) Its codomain is $(S^{\nvar})^{n} \subset \RR^{n(\nvar+1)}$. 
Let $\mbf{w}_{\ell} \in W_{j,\mu,\epsilon}$ 
($\ell=1,2$). Then, by \eqref{E:partial.Xi.bound} and Boothby \cite[Theorem (2.2), pp.\ 26--27]{wmB75}, we have
    \begin{equation*}
        \bigl| \Omega_{t}(\mbf{w}_{1}) - \Omega_{t}(\mbf{w}_{2}) \bigr| 
          \leq n \sqrt{ \nvar (\nvar+1) } K''' \sqrt{2(n-t)} \; |\mbf{w}_{1} - \mbf{w}_{2}| .
    \end{equation*}
I.e., $\Omega_{t}$ is Lipschitz on $W_{j,\mu,\epsilon}$ with Lipschitz constant proportional to $\sqrt{2(n-t)}$. 
By continuity, $\Omega_{t}$ is Lipschitz on $\overline{W}_{j,\mu,\epsilon}$ with Lipschitz constant proportional to $\sqrt{2(n-t)}$. 

Finally, let $\Xi_{t}(x) := \Omega_{t}(\mbf{w})$, 
where $x = (y_{1}, \ldots, y_{n}) \in \tilde{\Ss}_{a}$ satisfies \eqref{E:y.in.terms.of.w}. 
Write $\mbf{w} := (w_{1}, \ldots, w_{n}) \in  \bigl( \overline{B_{1}^{\nvar}(0)} \bigr)^{n}$. Then, by \eqref{E:xi(x)=x'.for.some.x.lambda} and \eqref{E:lambda.t.u.solves.eqn}, 
$\Xi(\tilde{\Ss}_{a}) = \tilde{\Ss}_{t}$. 

Fix $a \in (n-1,n)$ and let $t \in (a,n)$. Let $x \in \Ss_{a}$ and write $x$ as in \eqref{E:y.in.terms.of.w}. Then, by \eqref{E:St.wi.condn}, 
we have $\sum_{i=1}^{n} w_{i} = 0$. Let $x' = \Xi_{t}(x)$. Write
$x' = \bigl( \ldots, (w'_{i}, -\sqrt{1 - | w'_{i} |^{2} }), \ldots \bigr)^{1 \times n(\nvar +1)}$. Then, by \eqref{E:w'=lambda.w}, 
$\sum_{i=1}^{n} w'_{i} = 0$. I.e., by \eqref{E:St.wi.condn} again, $x' \in \Ss_{t}$. And conversely: $x'= \Xi_{t}(x) \in \Ss_{t}$ implies $x \in \Ss_{a}$. We conclude that 
$\Xi_{t}(\Ss_{a}) = \Ss_{t}$.

Now, $\mbf{w}$ is obtained from $x$ by projection, so $x \to \mbf{w}$ is Lipschitz, with Lipschitz constant 1. Let $\Ss_{a,j,\mu,\epsilon}$ be the set of $x \in \Ss_{a}$ s.t.\ \eqref{E:x.w.in.B} holds for $x$ and 
$\mbf{w} \in \overline{W}_{j,\mu,\epsilon}$. Then, $\Xi_{t}$ is Lipschitz on $\Ss_{a,j,\mu,\epsilon}$ with Lipschitz constant $K_{4} \sqrt{2(n-t)}$ 
for some $K_{4} < \infty$ (that does not depend on $t$ and only depending on $\mbf{w}$ through $a \bigl( \blds{\langle w \rangle}^{2} \bigr) = a(\mbf{u})$). 
Therefore, by \eqref{E:Lip.magnification.of.Hm} and \eqref{E:Sa.has.finite.vol}, we have
	\begin{equation*}
		\Hm^{n \nvar-\nvar-1} \bigl[ \Xi_{t}(\Ss_{a,j,\mu,\epsilon}) \bigr]
		  \leq \bigl[ \sqrt{2(n-t)} \bigr]^{n \nvar-\nvar-1} K_{4}^{n \nvar-\nvar-1} 
		    \Hm^{n \nvar-\nvar-1}(\Ss_{a,j,\mu,\epsilon}) < \infty .
	\end{equation*}
But $\Ss_{a} = \bigcup_{j} \Ss_{a,j,\mu,\epsilon}$ and $\Xi_{t}(\Ss_{a}) = \Ss_{t}$. Therefore, 
	\begin{equation*}
		\Hm^{n \nvar-\nvar-1}(\Ss_{t}) 
		  \leq 2 n \nvar \bigl[ \sqrt{2(n-t)} \bigr]^{n \nvar-\nvar-1} 
		    K_{4}^{n \nvar-\nvar-1} 
		      \Hm^{n \nvar-\nvar-1}(\Ss_{a}) < \infty ,
	\end{equation*}
where $2 n \nvar$ is the number of combinations $j,\mu,\epsilon$. By part \ref{I:aug.mean.essential.dist.approx} of the proposition (already proved), 
\eqref{E:R.n.nvar-nvar-1.H.bdd.above} now follows.
  	\end{proof}
  
  \begin{proof}[Proof of lemma \ref{L:solve.rho.t=r}] 
By \eqref{E:S.a.cap.T.empty}, \eqref{E:mu.a.Sa.is.compact}, and 
proposition \ref{P:rate.of.decrease.of.Hm.St.aug.direct.mean} part \ref{I:aug.mean.essential.dist.approx}, $R_{t} = \rho_{t}$ are both strictly positive, but go to 0 as $t \uparrow n$. And $R_{t} = \rho_{t}$ are locally Lipschitz, hence continuous. Therefore, at least if $r > 0$ is small, there exists $t_{r } \in [0,n)$ s.t. 
$R_{t_{r}} = \rho_{t_{r}} = r$. 

By \eqref{E:rho.t.approx.sqrt.2(n-t)}, $r = \sqrt{2(n-t_{r})} + O(n-t_{r})^{3/2}$. We solve this equation for $t_{r}$. Start with $r^{2} = 2(n- t_{r}) + O(n-t_{r})^{2}$. Let
    \begin{equation*}
      s :=1/(n - t_{r}) .
    \end{equation*} 
Thus, $r^{2} s^{2} = 2 s + O(1)$. Hence, there exists a positive constant 
$K < \infty$ s.t.
    \begin{equation}  \label{E:rsK.ineq}
       -K <  r^{2} s^{2} - 2 s < K .
    \end{equation}

By \eqref{E:mu.a.Sa.is.compact}, $\Ss_{\mu_{t_{r}}}$ is compact. By \eqref{E:directional.T.defn}, so is $\T$. Therefore, there exist 
$x \in \Ss_{\mu_{t_{r}}}$ and $y \in \T$ s.t.\  
Now, $\rho_{t_{r}} = \rho(x,y)$. (See \eqref{E:rho.is.top.metric}. So $x \neq y$.)  
$\rho(x,y) > |x-y|$, the Euclidean distance 
from $x$ to $y$. Hence, $\rho_{t_{r}} > |x-y|$ which in turn is no smaller than the Euclidean distance from $\Ss_{\mu_{t_{r}}}$ to $\T$. 
But, by proposition \ref{P:rate.of.decrease.of.Hm.St.aug.direct.mean}(\ref{I:aug.mean.euc.dist}),
the Euclidean distance from $\Ss_{\mu_{t_{r}}}$ to $\T$ is $\sqrt{2(n-t_{r})}$. Thus, $r = \rho_{t_{r}} > \sqrt{2(n-t_{r})}$, implying $r^{2} > 2(n - t_{r})$. 
I.e., $t_{r} > n - \tfrac{1}{2} r^{2}$. Therefore, 
    \begin{equation*}  
       n-t < \frac{1}{2} r^{2} 
       \Leftrightarrow \frac{2}{r^{2}} < \frac{1}{n-t_{r}} 
       \Leftrightarrow 2 < s r^{2} 
       \Leftrightarrow r^{2} s^{2} - 2 s > 0 .
    \end{equation*}
In summary,
    \begin{equation}   \label{E:n-t.<.(r.sqrd)/2}
        n-t < \frac{1}{2} r^{2} \text{ and } r^{2} s^{2} - 2 s > 0 .
    \end{equation} 

Let $f(s) := r^{2} s^{2} - 2 s$. From \eqref{E:rsK.ineq} and \eqref{E:n-t.<.(r.sqrd)/2}, we see that for $r, s$ of interest, $0 < f(s) < K$. Since $s > 0$, the solution to the equation 
$f(s) = K$ is 
    \begin{equation*}
      s_{K} := r^{-2} \bigl( 1 + \sqrt{1 + r^{2} K} \bigr) .
    \end{equation*} 
$f$ is convex with minimum at $s = r^{-2}$ and $f(r^{-2}) = -1/r^{2} < 0$. Since $K > 0$, we have $s_{K} >  2 r^{-2} >  r^{-2}$. Since $f$ is convex this means $f(s) \geq K$ 
if $s \geq s_{K}$. Hence, \eqref{E:rsK.ineq} implies $s < s_{K}$.

Consider the function 
    \begin{equation*}
      g(r) := \frac{K}{2} r^{2} - (r^{2} s_{K} - 2) 
        = \frac{K}{2}r^{2} -  \bigl[ (1 + \sqrt{1 + r^{2} K}) - 2 \bigr] 
          = \frac{K}{2} r^{2} -  \sqrt{1 + r^{2} K} + 1 .
    \end{equation*}
We have $g(0) = 0$ and $g'(r) = K r - K r/\sqrt{1 + r^{2} K} > 0$ for $r > 0$. Hence, $g(r) \geq 0$ for $r \geq 0$. Thus, since $s < s_{K}$, we have
    \begin{equation*}
      \frac{1}{n-t_{r}} = s < s_{K} \leq s_{K} + r^{-2} g(r) = s_{K} 
        + \left( \frac{K}{2} - s_{K} + 2 r^{-2} \right) = \frac{K}{2} + 2 r^{-2} .
    \end{equation*}
It follows that 
    \begin{equation*}
      t_{r} < n - \frac{r^{2}}{2} 
        + \left( \frac{r^{2}}{2} - \frac{r^{2}}{\frac{K}{2} r^{2} + 2} \right) < n .
    \end{equation*}
Now,
    \begin{equation*}
      \frac{r^{2}}{2} - \frac{r^{2}}{\frac{K}{2} r^{2} + 2} 
        = \frac{ (K/2) r^{4}}{K r^{2} + 4} = O(r^{4}) .
    \end{equation*}
Thus,
    \begin{equation*}  
        t_{r} < n - \frac{1}{2} r^{2} + O(r^{4}) . 
    \end{equation*}
Combining the preceding with \eqref{E:n-t.<.(r.sqrd)/2} we get 
    \begin{equation*}  
        t_{r} = n - \frac{1}{2} r^{2} + O(r^{4})  
    \end{equation*}
as desired.
  \end{proof}
 	
    \begin{proof}[Proof of lemma \ref{L:Pk.is.nhbd.retract}]
 If $w \in \RR^{\nvar+1}$ denote the $j^{th}$ coordinate of $w$ by $w^{j}$. Thus, 
	\begin{equation}  \label{E:S.nvar.coord.lwr.bnd}
		\text{if }  w \in S^{\nvar} \text{ then }  
		  \max_{j = 1, \ldots, \nvar+1} |w^{j}| \geq 1/\sqrt{\nvar+1}.
	\end{equation}

Let $k \in [0, n/2)$, so $n-k >k$. Let $x = (y_{1}, \ldots, y_{n} ) \in \D := (S^{q})^{n}$. 
For each $j = 1, \ldots, \nvar+1$ let $\tilde{y}^{j} = \tilde{y}^{j}(x)$ be the median 
of $y_{1}^{j}, \ldots, y_{n}^{j}$, i.e., the median of the $j^{th}$ coordinate of all the $y_{i}$'s. (If $n$ is even follow the usual convention of defining the median of $n$ numbers to be the midpoint between the two middle numbers when the numbers are arranged in increasing order.) 
Let $\tilde{y}(x) := (\tilde{y}^{1}, \ldots, \tilde{y}^{\nvar+1}) \in \RR^{\nvar+1}$. 
The map $x \mapsto \tilde{y}(x)$ is defined and continuous in $x \in \D$. 

Since the median is order invariant, we have 
	\begin{equation}  \label{E:y.tilde.Sn.invar}
		\tilde{y} \circ \sigma = \tilde{y} \text{ for every } \sigma \in S_{n}.
	\end{equation} 

Let $x = (y_{1}, \ldots, y_{n} ) \in \D$, not necessarily in $\Pf_{k}$, and, for 
$i \in \NN_{n}$, let 
	\begin{equation}  \label{E:ri(x).defn}
		r_{i}(x) = y_{i} - \tilde{y}(x) \in \RR^{q+1}.
	\end{equation} 
By \eqref{E:y.tilde.Sn.invar}, we have 
	\begin{equation}  \label{E:ri.Sn.equivar}
		r_{i} \bigl[ \sigma(x) \bigr] = y_{\sigma(i)} - \tilde{y} \bigl[ \sigma(x) \bigr] = y_{\sigma(i)} - \tilde{y}(x) = r_{\sigma(i)}(x),
			\qquad x \in \D, \;\sigma \in S_{n}.
	\end{equation} 
Let $\delta(x)$ be the $(k+1)^{st}$ largest value of $|r_{i}(x)|$ ($i \in \NN_{n}$). Thus, for at least $n-k$ indices $i$ we have $|r_{i}(x)| \leq \delta(x)$, for no more than $n-k-1$ indices $i$ can we have $|r_{i}(x)| < \delta(x)$, and for no more than $k$ indices $i$ can we have $|r_{i}(x)| > \delta(x)$. Note that $r_{1}(x), \ldots, r_{n}(x)$ and, hence, $\delta(x)$ are continuous in $x \in \D$. Note that, by \eqref{E:ri.Sn.equivar} we have 
	\begin{equation} \label{E:delta.Sn.invar}
		\delta \circ \sigma = \delta, \quad \sigma \in S_{n}.
	\end{equation} 

If $x \in \Pf_{k}$ then there exists at least one subset $J \subset \{ 1, \ldots, n \}$ s.t.\ $|J| = n-k$ ($|J|$ is the cardinality of $J$) and the points $y_{i} \in \Ss^{\nvar}$, $i \in J$, are identical. Let $y \in \RR^{\nvar+1}$ be the common value of $y_{i} \in \Ss^{\nvar}$, $i \in J$. (Because $|J| = n-k > n/2$, $y$ does not depend on $J$.) Then, since $k < n/2$ we have $\tilde{y}^{j}(x) = y^{j}$ ($j = 1, \ldots, \nvar+1$). This has two consequences. First, by \eqref{E:S.nvar.coord.lwr.bnd},
	\begin{equation*}
		\max_{j = 1, \ldots, \nvar+1} |\tilde{y}^{j}(x)| \geq 1/\sqrt{\nvar+1}.
	\end{equation*}
Second, at least $n-k$ of the $r_{i}(x)$'s are 0. Therefore, $\delta(x) = 0$ if $x \in \Pf_{k}$. Define
	\begin{equation*}
		\clU := \left\{ x \in \D : \,
		      max_{j = 1, \ldots, \nvar+1} \, |\tilde{y}^{j}(x)| > \frac{1}{2 \sqrt{\nvar+1} }
		        \text{ and } \delta(x) < \tfrac{1}{2} \right\}.
	\end{equation*}
Then $\clU$ is an open neighborhood of $\Pf_{k}$. By \eqref{E:y.tilde.Sn.invar} and \eqref{E:delta.Sn.invar} we have that $\sigma (\clU) = \clU$ for every $\sigma \in S_{n}$. We will define a retraction $R : \clU \to \Pf_{k}$. 

If $x \in \clU$, then, by definition of $\clU$, $\tilde{y}(x) \neq 0$. Now let
	\begin{multline}   \label{E:y.dot.defn}
		\dot{y}_{i}(x) = 
		\begin{cases}
		       \tilde{y}(x) = y_{i}, &\text{ if } r_{i}(x) = 0, \\
			\tilde{y}(x) + \max \left\{ 1 -\frac{ \delta(x)}{|r_{i}(x)|}, 0 \right\} r_{i}(x), 
			   &\text{ otherwise,}
              \end{cases} \\
		      \quad x \in \clU, \; i \in \NN_{n}.
	\end{multline}
(Thus, $\dot{y}_{i}(x) \in \RR^{q+1}$.) By \eqref{E:y.tilde.Sn.invar}, \eqref{E:ri.Sn.equivar}, and \eqref{E:delta.Sn.invar} we have
	\begin{equation} \label{E:y.dot.i.Sn.equivar}
		\dot{y}_{i} \bigl[ \sigma(x) \bigr] =\dot{y}_{\sigma(i)}(x), \quad x \in \D, \; \sigma \in S_{n}.
	\end{equation}

\emph{Claim:} $\dot{y}_{i}(x)$ is continuous in $x$. To see this, let $x_{0} \in \D$ and suppose $x \to x_{0}$ through $\D$. Then $\tilde{y}(x) \to \tilde{y}(x_{0})$,  $\delta(x) \to \delta(x_{0})$, and $r_{i}(x) \to r_{i}(x_{0})$. If $r_{i}(x_{0}) \neq 0$, then clearly $\dot{y}(x) \to \dot{y}(x_{0})$. So suppose $r_{i}(x_{0}) = 0$. Then $\dot{y}(x_{0}) = \tilde{y}(x_{0})$. We have $\dot{y}(x) = \tilde{y}(x) + \beta(x) r_{i}(x)$, where $0 \leq \beta(x) \leq 1$. Since $\tilde{y}(x) \to \tilde{y}(x_{0})$ and $r_{i}(x) \to 0$, we have $\dot{y}(x) \to \tilde{y}(x_{0}) = \dot{y}(x_{0})$. This completes the proof of the claim.

\emph{Claim:} If $x \in \clU$ then $\dot{y}_{i}(x) \neq 0$ ($i \in \NN_{n}$). For suppose that for some $i \in \NN_{n}$ we have $\dot{y}_{i} := \dot{y}_{i}(x) = 0$. If $y_{i} = \tilde{y} := \tilde{y}(x)$ then, by \eqref{E:ri(x).defn}, $r_{i} = 0$ so $\dot{y}_{i} = \tilde{y} \neq 0$ by definition of $\clU$. So suppose
$y_{i} \neq \tilde{y}$. I.e., $r_{i}(x) \neq 0$. In fact, $|y_{i} - \tilde{y}| = |r_{i}(x)| > \delta(x)$, because otherwise, by \eqref{E:y.dot.defn}, $y_{i} = \tilde{y}$. Thus, $0 = \dot{y} = \tilde{y} + \beta r_{i}(x)$, 
where $\beta = 1 - \delta(x)/\bigr| r_{i}(x) \bigl|$. Rewriting, 
	\begin{equation*}
		0 = \dot{y} = \tilde{y} + (y_{i} - \tilde{y}) - \frac{\delta(x)}{|y_{i} - \tilde{y}|} (y_{i} - \tilde{y}) 
		  = y_{i} - \frac{\delta(x)}{|y_{i} - \tilde{y}|} (y_{i} - \tilde{y}).
	\end{equation*}
Thus, 
	\begin{equation*}
		\delta(x) = \left| \frac{\delta(x)}{|y_{i} - \tilde{y}|} (y_{i} - \tilde{y}) \right| = |y_{i}| = 1.
	\end{equation*}
But $x \in \clU$ which means that $\delta(x) < 1/2$. This contradiction proves the claim that no 
$\dot{y}_{i}(x)$ is 0.

For $x \in \clU$ define $\hat{y}_{i}(x) = |\dot{y}_{i}(x)|^{-1} \dot{y}_{i}(x) \in S^{\nvar}$ and 
let 
    \begin{equation*}
      R(x) = \bigl( \hat{y}_{1}(x), \ldots, \hat{y}_{n}(x) \bigr) \in \D .
    \end{equation*} 
By \eqref{E:y.dot.i.Sn.equivar}, we have that $\rho \circ \sigma = \sigma \circ \rho$ on $\clU$ for every $\sigma \in S_{n}$. Moreover, $\rho$ is continuous on $\clU$, since the $\dot{y}_{i}$'s are. Furthermore, for the $n-k$ or more indices $i$ for which $|r_{i}(x)| \leq \delta(x)$ we have $1 -\delta(x)/|r_{i}(x)| < 0$ so $\hat{y}_{i}(x) = \tilde{y}(x)$. Thus, in fact $\rho : \clU \to \Pf_{k}$. If $x \in \Pf_{k}$ we have $\delta(x) = 0$ so, by \eqref{E:y.dot.defn} and \eqref{E:ri(x).defn}, 
$\dot{y}_{i}(x) = \tilde{y}(x) + 1 \times r_{i}(x) = \tilde{y}(x) + \bigl( y_{i} - \tilde{y}(x) \bigr) = y_{i}$. Now,  $|\dot{y}_{i}(x)| = |y_{i}| = 1$ so  $\rho(x) = x$.
  \end{proof} 
  
  \begin{proof}[Proof of corollary \ref{C:local.exactness.of.fit}]
Let $R : \clU \to \Pf_{k}$ be the retraction promised by the lemma. 
Let $x \in \clU$ and let 
$R(x) = (\hat{y}_{1}(x), \ldots, \hat{y}_{n}(x)) = (\hat{y}_{1}, \ldots, \hat{y}_{n}) \in \Pf_{k}$.  Recall, by \eqref{E:N.sub.n}, $\NN_{n} := \{1, \ldots, n \}$. Let $\mcl{I} = \mcl{I}(x) \subset 2^{\NN_{n}}$ be the collection of sets 
$I \subset \NN_{n}$ s.t.\ if $i, j \in I$ then $\hat{y}_{i} = \hat{y}_{j}$. 

Let $x \in \clU$ be arbitrary but fixed. Observe that if $I_{1}, I_{2} \in \mcl{I}$ 
and $I_{1} \cap I_{2} \neq \varnothing$, then $I_{1} \cup I_{2} \in \mcl{I}$. 
Let $J(x) \subset \{ 1, \ldots, n \}$ be a largest element in $\mcl{I}$. 
Since $R(x) \in \Pf_{k}$, the cardinality, $\bigl| J(x) \bigr|$, of $J(x)$ is at least $n-k > n/2$. But two subsets of $\NN_{n}$ of cardinality $> n/2$ must overlap. It follows that $J(x)$ is unique. 

Let $\nu_{R}(x) \in S^{\nvar}$ be the common value of all $\hat{y}_{i}$ with $i \in J(x)$. Since the restriction $R \restriction_{\Pf_{k}}$ is the identity on $\Pf_{k}$ and every data set in $\Pf_{k}$ has at least $n-k$ copies of the same point of $S^{\nvar}$, the map $\nu_{R}$ has order of exactness of fit $k$. 

\emph{Claim:} The map $J$ is ``upper semicontinuous'' in the sense that as $x' \to x$ ($x' \in \clU$) we eventually have $J(x') \subset J(x)$. To see this write 
$R(x') = (\hat{y}'_{1}, \ldots, \hat{y}'_{n}) \in \Pf_{k}$. 
Let $\epsilon = \min \Bigl\{ \bigl| \nu_{R}(x) - \hat{y}_{i} \bigr| : i \notin J(x) \Bigr\}$. 
By definition of $J(x)$, 
$\epsilon > 0$. Since $R$ is continuous, as $x' \to x$ in $\clU$, $\hat{y}'_{i} \to \hat{y}_{i}$ 
for every $i \in \NN_{n}$. Thus, eventually, $|\hat{y}'_{i} - \hat{y}_{i}| < \epsilon/3$ 
for every $i \in \NN_{n}$. Let $i \notin J(x)$ and $j \in J(x)$. Then 
    \begin{equation*}
      \epsilon < | \hat{y}_{i} - \hat{y}_{j} | 
        \leq | \hat{y}_{i} - \hat{y}'_{i} | + | \hat{y}'_{i} - \hat{y}'_{j} | + | \hat{y}'_{j} - \hat{y}_{j} |
          < \tfrac{1}{3} \epsilon + | \hat{y}'_{i} - \hat{y}'_{j} | + \tfrac{1}{3} \epsilon.
    \end{equation*}
Hence, $| \hat{y}'_{i} - \hat{y}'_{j} | > \tfrac{1}{3} \epsilon$. Suppose $J(x') \nsubseteq J(x)$ 
and let  $i \in J(x') \setminus J(x)$. 
Since $|J(x')| \geq n-k > n/2$ we have $J(x') \cap J(x) \neq \varnothing$. 
Let $j \in J(x') \cap J(x)$. Then $| \hat{y}'_{i} - \hat{y}'_{j} | > \tfrac{1}{3} \epsilon$. Contradiction. This proves the claim. 

It follows that $\nu_{R}$ is continuous: Let $\{x'_{m}\} \subset \clU$ converge to $x \in \clU$. Recall, by \eqref{E:N.sub.n}, $\NN_{n} := \{1, \ldots, n \}$. There exists a subsequence $\{x'_{m_{\ell}}\} \to x$ and $j \in \NN_{n}$ s.t.\ $j \in J(x'_{m_{\ell}})$ 
for every $\ell$. Thus, $\nu_{R} (x'_{m_{\ell}}) = \hat{y}_{j}(x'_{m_{\ell}})$. Hence, by continuity of $R$, we have $\nu_{R} (x'_{m_{\ell}}) \to \hat{y}_{j}(x) = \hat{y}_{j}$. But as we showed above, as $\ell \uparrow \infty$ eventually $j \in J(x)$ so $\hat{y}_{j} = \nu_{R}(x)$. 
I.e., $\nu_{R} (x'_{m_{\ell}}) \to \nu_{R}(x)$. The same argument works if, instead of $\{x'_{m}\}$, we had started with an arbitrary subsequence of it. Continuity of $\nu_{R}$ follows.

Notice that, if $\sigma \in S_{n}$ then, 
since $R \circ \sigma = \sigma \circ R$ on $\clU$, we have $J \circ \sigma = \sigma \circ J$. But, again since $| J(x) | \geq n-k > n/2$, we have that $\bigl[ \sigma \circ J(x) \bigr] \cap \bigl[ J(x) \bigr] \neq \varnothing$. Thus, $\nu_{R} \circ \sigma(x) = \nu_{R}(x)$. 
Hence, $\Phi := \nu_{R}$ satisfies \eqref{E:assume.Phi.S'.sym} on $\clU$.
  \end{proof}

\begin{proof}[Proof of lemma \ref{L:loose.exactness.of.location.fit}]
Let $\msf{V}_{\pi/2}$ be the cover on $\F := S^{\nvar}$ constructed in section \ref{S:convex.combos.on.spheres}, where a convex combination function on $\msf{V}_{\pi/2}$ is also constructed. Since $\Pf_{k}$ is closed and (by assumption) $\Phi$ is defined and continuous on $\Pf_{k}$, by lemma \ref{L:Pk.is.nhbd.retract} and theorem \ref{T:if.lin.combo.on.F.then.can.rstrct.to.bad.sings} part \ref{I:Omega.Phi.agree.on.Pf}, there exists a measure of location, $\Omega$, continuous 
on $\tilde{\D} := \D \setminus \Ss^{\msf{V}_{\pi/2}} \supset \D \setminus \Ss$ satisfying \eqref{E:assume.Phi.S'.sym} and agreeing with $\Phi$ on $\Pf_{k}$. Since $\Phi$ satisfies \eqref{E:D'.=.D.less.S} and \eqref{E:no.severe.sings.in.Pk} and, by \eqref{E:SV.is.closed}, 
$\Ss^{\msf{V}_{\pi/2}}$ is closed, we see that $\Omega$ satisfies \eqref{E:loc.measure.has.no.sings.near.T}. Therefore, making the $\clU$ in lemma \ref{L:Pk.is.nhbd.retract} smaller, if necessary, we have 
$\Ss^{\msf{V}_{\pi/2}} \cap \clU = \varnothing$. Hence, 
    \begin{equation*}
      \clU \subset \tilde{\D} \text{ and } \Omega \text{ is continuous in } \clU .
    \end{equation*}
 
Let $\nu_{R}$ be the local measure of location as in corollary \ref{C:local.exactness.of.fit}. 
Since $\nu_{R}$ has exactness of fit order $k$ and $\Omega$ and $\Phi$ agree 
on $\Pf_{k}$, \eqref{E:exactness.of.fit.loose} tells us that if $x \in \Pf_{k}$ then
	\begin{equation}   \label{E:Phi.dot.mu.rho.-1}
		 \Omega(x) \cdot \nu_{R}(x) > -1.
	\end{equation} 
Recall $S_{n}$ be the group of permutations of $\NN_{n} := \{1, \ldots, n \}$. If it is not already the case that $\sigma(\clU) = \clU$ for all $\sigma \in S_{n}$ (remember, we modified the $\clU$ from lemma \ref{L:Pk.is.nhbd.retract}), simply replace $\clU$ by $\bigcap_{\sigma} \sigma(\clU)$, where the intersection is taken over $\sigma \in S_{n}$. Hence, we may assume 
$\sigma(\D \setminus \clU) = \D \setminus \clU$ for every $\sigma \in S_{n}$. 

Define 
	\begin{equation}  \label{E:arg!.maties!2}
	   \arg(y) := |y|^{-1} y \in S^{\nvar} \quad 
		  \bigl( y \in \RR^{\nvar +1} \setminus \{ 0 \} \bigr).
	\end{equation} 
(See \eqref{E:arg!.maties!} and \eqref{E:arg,matey!.2}) By Urysohn's Lemma (Simmons \cite[Theorem A, p.\ 135]{gfS63}) there exists a continuous function 
$f : \D \to [0,1]$ s.t.\ $f = 1$ on $\Pf_{k}$ and $f = 0$ on $\D \setminus \clU$. We may assume $f$ is symmetric in its arguments, else replace $f$ by $(n!)^{-1} \sum_{\sigma \in S_{n}} f \circ \sigma$. 
\emph{Claim:} $\arg \Bigl( \bigl[ 1 - f(x) ] \Omega(x) + f(x) \nu_{R}(x) \Bigr)$ is defined for 
$x \in \clU \subset \tilde{\D}$. Here, we regard $\Omega(x)$ and $\nu_{R}(x)$ as points in $\RR^{\nvar+1}$. This is equivalent to showing   
	\begin{equation}   \label{E:Phi.arg.combo.not.0}
	  \Bigl| \bigl[ 1 - f(x) ] \Omega(x) + f(x) \nu_{R}(x) \Bigr| > 0 
	      \text{ if } x \in \clU \cap \tilde{\D}.
	\end{equation}
Now, \eqref{E:Phi.arg.combo.not.0} obviously holds if $f(x) = 0$ or 1. Assume $0 <  f(x) < 1$. Then, by \eqref{E:Phi.dot.mu.rho.-1},
	\begin{multline}  \label{E:norm.of.Phi.arg.combo}
		\Bigl| \bigl[ 1 - f(x) ] \Omega(x) + f(x) \nu_{R}(x) \Bigr|^{2}  \\
		\begin{aligned}
			&= \bigl[ 1 - f(x) ]^{2} + 2 \bigl[ 1 - f(x) ] f(x) \, 
			        \Omega(x) \cdot \nu_{R}(x) + f(x)^{2} \\
			&> \bigl[ 1 - f(x) ]^{2} - 2 \bigl[ 1 - f(x) ] f(x)  + f(x)^{2} \\
			&= \bigl[ 1 - 2 f(x) \bigr]^{2} \geq 0.
		\end{aligned}
	\end{multline} 
Thus, \eqref{E:Phi.arg.combo.not.0} holds.

Let $\mu : \tilde{\D} \to S^{\nvar}$ by 
   \[
      \mu(x) = 
         \begin{cases}
            \arg \Bigl( \bigl[ 1 - f(x) ] \Omega(x) + f(x) \nu_{R}(x) \Bigr),
               \text{ if } x \in \clU, \\ 
            \Omega(x), \text{ if } x \in \tilde{\D} \setminus \clU.
         \end{cases}
   \]
By \eqref{E:Phi.arg.combo.not.0}, $\mu(x)$ is defined and continuous everywhere 
on $\tilde{\D} \supset \D \setminus \Ss$. 
Moreover, $\mu$ is clearly symmetric in its arguments, i.e. it satisfies \eqref{E:assume.Phi.S'.sym}, since $f$, 
$\Omega$, and $\nu_{R}$ are and $\clU$ is $S_{n}$-invariant. $\mu$ has order of exactness of fit $k$ because $\nu_{R}$ does and $f = 1$ on $\Pf_{k}$. $\nu_{R}$ is continuous 
on $\clU$ and $\Omega$ is continuous off $\Ss^{\msf{V}_{\pi/2}}$. Therefore, $\mu$ is continuous off $\Ss^{\msf{V}_{\pi/2}}$. By assumption, $\Phi$ satisfies \eqref{E:loc.measure.has.no.sings.near.T}. \emph{A fortiori}, so does $\mu$ and 
$\Ss = \Ss^{\msf{V}_{\pi/2}}$.

Since $\nu_{R}$ has order of exact fit $k$, by \eqref{E:exact.fit.then.on.diagnl.exact} it and hence $\mu$ satisfies \eqref{E:Phi.on.diagnl.exact}. Since the singular set 
of $\mu$ lies in $\msf{V}_{\pi/2}$, theorem \ref{T:spher.loc.sing.codim.nvar+1} applied to $(\mu, \Ss^{\msf{V}_{\pi/2}})$ implies \eqref{E:poz.measure.for.exact.fit.sings}.
  \end{proof}

  \begin{proof}[Proof of proposition \ref{P:aug.mean.sing.set.like.exact.fit}] 
Let $k = 0$. By \eqref{E:basic.props.of.mu.a}, $\mu_{a}$ satisfies \eqref{E:loc.measure.has.no.sings.near.T}. Since, as observed above, $\Pf_{0} = \T$, \eqref{E:no.severe.sings.in.Pk} holds with $k=0$. Therefore, the proposition holds for the case $k=0$. 

So let $k \in (0, n/2)$ and 
    \begin{equation*}
      a \in [0, n-2k) .
    \end{equation*} 
We show that there exists a measure of location on $S^{\nvar}$ with order of exactness of fit $k$ whose singular set is a subset of $\Ss_{a}^{\msf{V}_{\pi/2}}$. By \eqref{E:basic.props.of.mu.a}, $\mu_{a}$ satisfies \eqref{E:assume.Phi.S'.sym} and, by \eqref{E:mu.a.satisfies.E:D'.dense}, $\mu_{a}$ satisfies \eqref{E:D'.=.D.less.S}. Therefore, by lemma \ref{L:loose.exactness.of.location.fit}, 
it suffices to show that $\mu_{a}$ satisfies \eqref{E:no.severe.sings.in.Pk} and \eqref{E:exactness.of.fit.loose}.

Let $y_{0} \in S^{\nvar}$ be the augmentation point for $\mu_{a}$. Let $y_{n-k+1}, \ldots, y_{n} \in S^{\nvar}$ be arbitrary but fixed. Let $w := a y_{0} + y_{n-k+1} + \cdots + y_{n} \in \RR^{\nvar+1}$. Then
	\begin{equation}   \label{E:a+k.w.bnd}
		|w| \leq a + k < B := n-k.
	\end{equation}
Thus, for any $y \in S^{\nvar}$, we have
	\begin{equation} \label{E:|w+By|>0}
		\bigl| a y_{0} + (n-k) y + y_{n-k+1} + \cdots + y_{n} \bigr| 
			= | w + By | \geq B - |w| > 0. 
	\end{equation}
So, by \eqref{E:Sa=Sa'} and \eqref{E:S'a.D'a.defns}, $\mu_{a}$ has no singularities 
in $\Pf_{k}$. Thus, $\Phi := \mu_{a}$ satisfies \eqref{E:no.severe.sings.in.Pk}. In particular, $\mu_{a}$ is defined and continuous on $\Pf_{k}$. We will show that 
$\mu_{a}(y, \ldots, y, y_{n-k+1}, \ldots, y_{n}) \cdot y > -1$ for every 
$y \in S^{\nvar}$. Since $y_{n-k+1}, \ldots, y_{n}$ are arbitrary, that would imply \eqref{E:exactness.of.fit.loose} holds for $\mu_{a}$ so, by lemma \ref{L:loose.exactness.of.location.fit}  there exists a measure of location on $S^{\nvar}$ with order of exactness of fit $k$ whose singular set is a subset of that of $\mu_{a}$.

Let $y \in S^{\nvar}$. We can write $y = \alpha w + z$, where $\alpha \in \RR$ ands $z \in \RR^{\nvar+1}$ with $z \perp w$. 
(So if $w = 0$, $z = y$.) Let 
    \begin{equation*}
      A = |w|^{2} \geq 0. 
    \end{equation*}
Thus, $\alpha^{2} A + |z|^{2} = 1$, so
	\begin{equation}  \label{E:norm.alpha.sqrt.A}
		1 = |y| = | \alpha w + z | \geq |\alpha| \sqrt{A},
	\end{equation}
and
	\begin{equation*}
		1 \geq \mu_{a}(y, \ldots, y, y_{n-k+1}, \ldots, y_{n}) \cdot y 
		  = |w + B y|^{-1} (w + B y) \cdot y
		    = \frac{\alpha A + B}{\sqrt{A + 2 \alpha AB + B^{2}}}.
	\end{equation*}
Thus, we need to show 
	\begin{equation}  \label{E:A.B.sqrt.thingy.>.-1}
		f(\alpha) := \frac{\alpha A + B}{\sqrt{A + 2 \alpha AB + B^{2}}} > -1.
	\end{equation}
By \eqref{E:|w+By|>0}, $\bigl| f(\alpha) \bigr| > 0$. Now,
	\begin{equation*}
		f(\alpha) = \frac{\alpha A + B}{\sqrt{(\alpha A + B)^{2} + (A - \alpha^{2} A^{2})}}.
	\end{equation*}
If $A=0$, then $f(\alpha) = B/\sqrt{B^{2}} = 1 > -1$, since $B := n-k > 0$. Thus,  \eqref{E:A.B.sqrt.thingy.>.-1} holds if $A=0$.
\eqref{E:A.B.sqrt.thingy.>.-1} also holds if $\alpha A + B \geq 0$, 
in which case $f(\alpha) > 0 > -1$; or if $A - \alpha^{2} A^{2} > 0$, 
in which case $f(\alpha) \geq - \bigl| f(\alpha) \bigr| > -1$. 

By \eqref{E:norm.alpha.sqrt.A}, 
	\begin{equation*}
		A - \alpha^{2} A^{2} = A(1 - \alpha^{2} A)
		  = A \Bigl[ |y|^{2} - \bigl( |\alpha| \sqrt{A} \bigr)^{2} \Bigr] 
		    \geq 0 
	\end{equation*}
As just observed, if this inequality is strict, then \eqref{E:A.B.sqrt.thingy.>.-1} is proved. 
As already observed, if $A = 0$, \eqref{E:A.B.sqrt.thingy.>.-1} is proved again. 

So assume $A - \alpha^{2} A^{2} = 0$ but $A > 0$. ($A := |w|^{2}$, so $A \geq 0$.) 
Then $1 - \alpha^{2} A = 0$ (so $z = 0$). 
Hence, $\alpha A = \pm \sqrt{\alpha^{2} A} \sqrt{A} = \pm \sqrt{A}$. Thus, by \eqref{E:a+k.w.bnd},
$\alpha A + B \geq B - \sqrt{A} = (n-k) - |w| > 0$ and \eqref{E:A.B.sqrt.thingy.>.-1} is proved: $f(\alpha) = (\alpha A + B)/|\alpha A + B| = 1 > -1$. ($A - \alpha^{2} A^{2} = 0$.) This concludes the proof of \eqref{E:exactness.of.fit.loose} when $a \in [0, n - 2k)$. 

Next, \emph{suppose} $a \in [n-2k, n)$. Then $k > 0$. Recall that $y_{0}$ is the augmentation point for $\mu_{a} = \mu_{y_{0}, a}$. Let 
    \begin{equation*}
         v \in S^{\nvar} \text{ be perpendicular to } y_{0}.
    \end{equation*}
\emph{First, SUPPOSE $k$ IS EVEN.} Notice that $a \in [n-2k, n)$ implies that 
$\bigl| a - (n-k) \bigr| \leq k$. Recycle ``$A$'', redefining it as follows. 
$A := \tfrac{a - (n-k)}{k}$. So $1 - A^{2} \geq 0$. Let 
    \begin{equation} \label{E:w.sqrt.v.k.frac}
       w := \sqrt{1 - A^{2}} \; v.
    \end{equation} 
So $-A y_{0} \pm w \in S^{\nvar}$. Set half of $y_{n-k+1}, \ldots, y_{n}$ equal 
to $-A y_{0} + w$ and half equal to $-A y_{0} - w$. Set $y = -y_{0}$. Then
	\begin{equation} \label{E:Pk.singularity}
		(y, \ldots, y, y_{n-k+1}, \ldots, y_{n}) \in \Pf_{k} 
		   \text{ and } a y_{0} + (n-k) y + \sum_{i=n-k+1}^{n} y_{i} = 0.
	\end{equation}
So, by \eqref{E:Sa=Sa'} and \eqref{E:S'a.D'a.defns}, $\mu_{a}$ has a singularity in $\Pf_{k}$. Note that multiple singularities can be obtained this way since we can set $y_{n-k+1}, \ldots, y_{n}$ equal to $-A y_{0} + w$ and half equal to $-A y_{0} - w$ in more than one way.

\emph{LET $k > 1$ BE ODD.} Then $k \geq 3$, so $n > 6$ by \eqref{E:n.nvar.k.sizes}, and $n - 2k + 2 < n$. First, 
    \begin{equation*}
       \text{Suppose that } a \in [n-2k+2, n). 
    \end{equation*}
We show 
	\begin{equation}  \label{E:|a-(n-k+1)|}
		\bigl| a - (n-k+1) \bigr| \leq k-1,
	\end{equation} 
so $1 - \left[ \tfrac{a - (n-k+1)}{k-1} \right]^{2} \geq 0$. To prove \eqref{E:|a-(n-k+1)|}, 
first suppose $a - (n-k+1) \geq 0$, then $a < n$ implies 
$\bigl| a - (n-k+1) \bigr| = a - (n-k+1) < k-1$. Next, suppose $a - (n-k+1) < 0$. 
Then $\bigl| a - (n-k+1) \bigr| = (n-k+1) - a \leq (n-k+1) - (n-2k+2) = k-1$. So \eqref{E:|a-(n-k+1)|} holds. 

We now proceed in a fashion similar to that we used for the $k$ even case. Let $v \in S^{\nvar}$ be perpendicular to $y_{0}$ and let $w$ be defined as in \eqref{E:w.sqrt.v.k.frac}. Set $y = -y_{0}$. This time let 
    \begin{equation} \label{E:A.k-1.denom.defn}
       A := \tfrac{a - (n-k+1)}{k-1}, \text{ so } |A| \leq 1.
    \end{equation}
Then $|-A y_{0} \pm w| = 1$. Since $k \geq 3$, we have $k-1\geq 2 > 0$ is even. Set half of $y_{n-k+1}, \ldots, y_{n-1} \in S^{\nvar}$ equal to $-A y_{0} + w$ and half equal to $-A y_{0} - w$. Thus, $y_{n-k+1}, \ldots, y_{n-1} \in S^{\nvar}$. Set $y_{n} := y := -y_{0}$. Once again \eqref{E:Pk.singularity} holds.

Now suppose 
    \begin{equation*}
        a \in [n-2k, n-2k+2) .
    \end{equation*}
First, we prove
	\begin{equation*}  \label{E:|a-(n-k-1)|}
		\bigl| a - (n-k-1) \bigr| \leq k-1,
	\end{equation*} 
so $1 - \left[ \tfrac{a - (n-k-1)}{k-1} \right]^{2} \geq 0$. Since $k \geq 3$, 
we have $a - (n-k-1) < (n-2k+2) - (n-k-1) = -k+3 \leq 0 < k-1$. 
Similarly, $(n-k-1) - a \leq (n-k-1) - (n-2k) = k-1$. 

This time let $A := \tfrac{a - (n-k-1)}{k-1}$. Now proceed in the usual way: 
Let $w := \sqrt{1 - A^{2}} \, v$. 
Set half of $y_{n-k+1}, \ldots, y_{n-1}$ equal 
to $-A y_{0} + w$ and half equal to $-A y_{0} - w$. Set $y := -y_{0}$ and $y_{n} := +y_{0}$. Once again \eqref{E:Pk.singularity} holds.

Finally, WE CONSIDER THE CASE 
    \begin{equation*}
       k=1 \text{ and } a \in [n-2k, n) = [n-2, n).
    \end{equation*} 
Let $m := (n-1)^{2}$. Recall we still assume \eqref{E:n>2,nvar>0}: $n > 2$. 
So $m \geq 4$ and $1 - m - a^{2} < 0$. 
Let 
   \begin{equation*}
      c := \frac{1 - m - a^{2}}{2(n-1) a} < 0.
   \end{equation*}
Since $a \in [n-2, n)$ we have $1 \geq \bigl[ (n-1)-a \bigr]^{2}$. 
Thus, $1 - m - a^{2} \geq -2(n-1)a$. Dividing both sides by $2(n-1)a > 0$, we get that $c \geq -1$. Since $c < 0$ this means 
	\begin{equation}  \label{E:|c|.leq.1}
		|c| \leq 1.
	\end{equation}
 In particular, $m (1 - c^{2}) \geq 0$.

Now,
	\begin{equation*}
		\frac{4 a^{2} m - (1 - 2m - 2 a^{2} + m^{2} + 2 a^{2} m + a^{4})}{4 a^{2}} 
		  = m \left( 1 - \frac{[ 1 - m - a^{2} ]^{2}}{4m a^{2}} \right) = m (1-c^{2}) .
	\end{equation*}
Therefore,
	\begin{align}  \label{E:4a.sqrd.(n-1).sqrd.etc}
		4 a^{2} \bigl[ 1 - m (1-c^{2}) \bigr] &= 
		          4 a^{2} - 4 a^{2} m \notag \\
		               &\qquad \qquad 
		                 + ( 1 - 2m - 2 a^{2} + m^{2} + 2 a^{2} m + a^{4} ) \notag \\
		     &= 1 - 2m + 2 a^{2} + m^{2} - 2 a^{2} m + a^{4} \\
		     &= (1 - m)^{2} 
		          + 2 a^{2} (1 - m) + a^{4} \notag \\
		     &= (1 - m + a^{2})^{2} \geq 0. \notag 
	\end{align}
Thus,
	\begin{equation}  \label{E:(n-1)c+ a}
		      \bigl[ (n-1)c+ a \bigr]^{2} 
		      = \left( \frac{1 - m + a^{2}}{2a} \right)^{2} = 1 - m (1-c^{2}).
	\end{equation}
\eqref{E:4a.sqrd.(n-1).sqrd.etc} also implies
	\begin{equation} \label{E:(n-1)sqrd.(1 - c.sqrd).leq.1}
		m (1 - c^{2}) \leq 1.
	\end{equation}
	
Now let $v \in S^{\nvar}$ be perpendicular to $y_{0}$. Recall that, by \eqref{E:|c|.leq.1}, $|c| \leq 1$. Let
	\begin{equation*}
		y := c y_{0} + \sqrt{1 - c^{2}} \; v \text{ and }
			y_{n} := \pm \sqrt{1 - m (1-c^{2})} \; y_{0} - (n-1) \sqrt{1 - c^{2}} \; v .
	\end{equation*}
Since $m := (n-1)^{2}$, we have $y, y_{n} \in S^{\nvar}$.  Consider the point 
$(y, \ldots, y, y_{n}) \in \Pf_{k}$. (So again $y$ is repeated $n-k = n-1$ times.) 
Now,
$a y_{0} + (n-1) y + y_{n} = a y_{0} + (n-1) c y_{0} \pm \sqrt{1 - m (1-c^{2})} \; y_{0}$. Therefore, by choosing the sign in $y_{n}$ appropriately, \eqref{E:(n-1)c+ a} tells us that
	\begin{equation*}
		a y_{0} + (n-1) y + y_{n} =  0.
	\end{equation*}
I.e., $(y, \ldots, y, y_{n}) \in \Pf_{k}$ is a singularity of $\mu_{a}$ when $k=1$ 
and $a \in [n-2, n)$.
  \end{proof}

  \begin{proof}[proof of lemma \ref{L:not.both.antipodals.in.Ma}]
Suppose $z_{\beta} = - z_{\alpha}$  and $z_{\alpha} \in M_{a}$. 
We show $z_{\beta} \notin M_{a}(x)$. WLOG $\alpha = 1$, $\beta = 2$. 
Suppose $z_{2} \in M_{a}(x)$. We will get a contradiction.

By \eqref{E:Ga.defn}, 
    \begin{equation}  \label{E:Ga.1,2}
      G_{a}(v; x) = G_{a}(v) = a \angle(v, y_{0}) + \ell_{1} \angle(z_{1},v) 
        + \ell_{2} \angle(z_{2},v) 
          + \sum_{\omega=3}^{t} \ell_{\omega} \angle (z_{\omega}, v) ,
            \quad (v \in S^{1}) .
    \end{equation}
(The summation at the end is 0 if $t = 2$.) We will differentiate this w.r.t.\ $v$ at $v = z_{1}$ and $v = z_{2}$. So the $z_{i}$'s remain fixed while $v$ varies. Let $w \in S^{1}$ be perpendicular to $z_{1}$, which means $w \perp z_{2}$ as well. Let $\gamma = 1, 2$. 
Let $\phi : (-\pi, \pi] \to S^{1}$ be a parametrization of $S^{1}$ by arc length 
with $\phi(0) = z_{\gamma}$ and $\phi(\pi/2) = w$. 
Thus, $\phi(\pi) = z_{3-\gamma}$. Let $z \in M_{a}(x)$. 
If $z \neq \pm z_{\gamma}$, then, by \eqref{E:angle.deriv.sign}, 
$(d/ds) \angle \bigl( z,\phi(s) \bigr) \restriction_{s=0}$ exists and equals $-sign(z \cdot w)$. 

Since $z_{\gamma} \in M_{a}(x)$, as $v = \phi(s) \in S^{1}$ approaches $z_{\gamma}$ from either direction, $G_{a}(v ;x)$ must eventually be non-increasing. In fact, it must strictly decrease as $s \uparrow 0$ (this corresponds 
to $\tfrac{d-}{dv} G(v;x) \restriction_{v=z} < 0$). Because, if $G_{a}(v ;x)$ were constant on one side of $z$ or the other then $M_{a}(x)$ would contain points not in $Y(x)$, contradicting \eqref{E:Ma(x).in.Y}. 
Similarly, $G_{a}\bigl ( \phi(s); x \bigr)$ must strictly increase as $s > 0$ 
pulls away from $0$ ($\tfrac{d+}{dv} G(v;x) \restriction_{v=z} > 0$). 
To sum up, if $\gamma = 1$ or 2, so $z_{\gamma} \in M_{a}(x)$, we have
    \begin{equation}  \label{E:one.sided.derivs.of.Ga}
      \frac{d_{+}}{ds} G_{a} \bigl ( \phi(s); x \bigr) \restriction_{s=0} \, > 0
        \text{ and } \frac{d_{-}}{ds} G_{a} \bigl ( \phi(s); x \bigr) \restriction_{s=0} \, < 0 .
    \end{equation}

First, assume neither $z_{1}$ nor $z_{2}$ equals $y_{0}$. Since $z_{2} = - z_{1}$, if $\gamma = 1$ or 2, by \eqref{E:angle.deriv.sign} the only terms in \eqref{E:Ga.1,2} which are not differentiable in $v$ 
at $v = z_{\gamma} \in M_{a}$ are $\ell_{\gamma} \angle(v, z_{\gamma})$ and 
$\ell_{3-\gamma} \angle(v, z_{3-\gamma})$. Take $\gamma = 1$. Then, by \eqref{E:Ga.1,2}, \eqref{E:one.sided.derivs.of.Ga}, \eqref{E:angle.deriv.sign}, and \eqref{E:one-sided.angle.derivs}, 
    \begin{align*}
      0 &<  \frac{d_{+}}{ds} G_{a} \bigl ( \phi(s); x \bigr) \restriction_{s=0} \, 
        = - a \, sign(y_{0} \cdot w) 
          + \ell_{1} \frac{d_{+}}{ds} \angle \bigl( z_{1}, \phi(s) \bigr) \restriction_{s=0} \\
         & \qquad + \ell_{2} \frac{d_{+}}{ds} \angle \bigl( z_{2}, \phi(s) \bigr) 
           \restriction_{s=0} \, 
              - \sum_{\omega=3}^{n} \ell_{\omega} \, sign(z_{\omega} \cdot w) \\
         &= - a \, sign(y_{0} \cdot w) + (\ell_{1} - \ell_{2}) 
              - \sum_{\omega=3}^{n} \ell_{\omega} \, sign(z_{\omega} \cdot w) .
    \end{align*}
Similarly,
    \begin{equation*}
      0 >  \frac{d_{-}}{ds} G_{a} \bigl ( \phi(s); x \bigr) \restriction_{s=0} \, =
        - a \, sign(y_{0} \cdot w) - (\ell_{1} - \ell_{2}) 
            - \sum_{\omega=3}^{n} \ell_{\omega} \, sign(z_{\omega} \cdot w) .
    \end{equation*}
It follows that $\ell_{1} > \ell_{2}$. 

Now take $\gamma = 2$ so now $\phi(0) = z_{\gamma} = z_{2}$. 
Then we conclude $\ell_{2} > \ell_{1}$. Contradiction. That proves lemma for the case when neither $z_{1}$ nor $z_{2}$ equal $y_{0}$.

The case $z_{\gamma} = y_{0}$ is similar. If $\gamma = 1$ then      \begin{align*}
      0 &<  \frac{d_{+}}{ds} G_{a} \bigl ( \phi(s); x \bigr) \restriction_{s=0} \, 
        = (a + \ell_{1} )\frac{d_{+}}{ds} \angle(z_{1},v) \restriction_{v=z_{1}} \\
         & \qquad - \ell_{2} \frac{d_{+}}{ds} \angle(z_{2},v) \restriction_{v=z_{1}} 
              - \sum_{\omega=3}^{n} \ell_{\omega} \, sign(z_{\omega} \cdot w) \\
         &= (a + \ell_{1} - \ell_{2}) 
              - \sum_{\omega=3}^{n} \ell_{\omega} \, sign(z_{\omega} \cdot w)  
              \intertext{ and } 
       0 &>  \frac{d_{-}}{ds} G_{a} \bigl ( \phi(s); x \bigr) \restriction_{s=0} \, 
          = -(a + \ell_{1} - \ell_{2}) 
            - \sum_{\omega=3}^{n} \ell_{\omega} \, sign(z_{\omega} \cdot w) .
        \end{align*} 
We conclude $a + \ell_{1} - \ell_{2} > 0$. With $\gamma = z_{2}$ we get the opposite. This contradiction concludes the proof.
 \end{proof}

  \begin{proof}[Proof of lemma \ref{L:D'.is.dense.for.aug.median}]
Let $z_{1}, \ldots, z_{t} \in S^{1}$ be the distinct locations of the points 
$y_{1}, \ldots, y_{n}$. We only count observations, i.e.\ $y_{1}, \ldots, y_{n}$, not the augmentation point $y_{0}$, but it is possible that 
$y_{0} \in \{ z_{1}, \ldots, z_{t} \}$. 
Let $\ell_{\alpha} \in [1,n]$ be the multiplicity 
of $z_{\alpha}$ ($\alpha=1, \ldots, t$) 
so $\ell_{1} + \cdots + \ell_{t} = n$. Call $\{ z_{1}, \ldots, z_{t} \}$ the ``support'' of $x$. 
If for every $j \in \NN_{n}$ we have $y_{j} \neq y_{0}$, define $z_{0} = y_{0}$. 

Suppose $x = ( y_{1}, \ldots, y_{n} ) \in \D \setminus \D'$, 
then there exist $z_{i}, z_{j} \in M_{a}(x) \subset Y(x)  := \{ y_{0}, y_{1}, \ldots, y_{n} \}$,
but $z_{i} \neq z_{j}$. Let 
    \begin{equation*}
      g := g(x) := \min_{v \in S^{1}} G_{a}(v;x) . 
    \end{equation*}
Choose $z_{i}, z_{j} \in M_{a}(x)$ to have minimal separation in terms of arc length of any pair of points in $M_{a}(x)$. Thus, $t > 1$ and $G_{a}(z_{j};x) = G_{a}(z_{i};x) = g$. Either $i$ or $j$ may be 0. WLOG $j = 1$. $z_{i} \neq z_{1}$ and, by lemma \ref{L:not.both.antipodals.in.Ma}, $z_{i} \neq -z_{1}$. We may always take 
$z_{1} \neq y_{0}$ because if $z_{1} = y_{0}$ we may simply swap the indices $i$ and 1.

Let $u, v \in S^{1}$ with $v \notin \{u,-u\}$. The two arcs in $S^{1}$ joining $u$ and $v$ have positive length and, since $v \neq u$ and $v \neq -u$, one arc has length $< \pi$, the other $> \pi$. Say that a point of $S^{1}$ lies ``between $u$ and $v$'' if it lies in the shorter arc. Thus, $u$ and $v$ lie between $u$ and $v$. A point lying in the interior of the shorter arc lies ``strictly between'' $u$ and $v$ so neither $u$ nor $v$ lie strictly between $u$ and $v$

Let 
$B := \bigl\{ z \in Y(x) \cup (-Y(x)) : z \text{ lies strictly between } z_{i} \text{ and } z_{1} \bigr\}$. (See \eqref{E:Y(x).defn}.) \emph{Claim:} $B \neq \varnothing$. 
For suppose $B = \varnothing$. Then obviously no point of $Y(x)$ or $-Y(x)$ lies strictly between $z_{i}$ and $z_{1}$. If $z \in Y(x) \cup (-Y(x))$ lay strictly between $-z_{i}$ and $-z_{1}$ then $-z \in Y(x) \cup (-Y(x))$ and lies between $z_{i}$ and $z_{1}$ so that is also ruled out. We conclude that 
$B = \varnothing$ implies any $z \in Y(x) \cup (-Y(x))$ lies between $z_{1}$ and $-z_{i}$ or between $z_{i}$ and $-z_{1}$. (Draw a picture.)

Continue to suppose $B = \varnothing$. WLOG $z_{1} = (1,0)$. Measure angles counterclockwise from $z_{1}$. Then, since $z_{i} \neq - z_{1}$, we may assume 
$\theta := \angle(z_{i}, z_{1}) \in (0, \pi)$. Let $v = (\cos \phi, \sin \phi) \in S^{1}$ lie strictly between $z_{i}$ and $z_{1}$. Thus, we may take $\phi \in (0, \theta)$. 
Take $w := w(v) := \bigl( \cos (\phi+\pi/2), \sin (\phi+\pi/2) \bigr) \in S^{1}$, 
so $w \perp v$. 
If $z = (\cos \zeta, \sin \zeta)  \in S^{1}$, then 
$z \cdot w = \cos (\zeta - \phi - \pi/2) = -\sin(\zeta - \phi)$. Suppose $z$ lies between $z_{1}$ and $-z_{i}$. We may take $\zeta \in [-\theta, 0]$. I.e., both $\zeta$ and $\phi$ lie in the interval $[-\theta, \theta]$. Hence, $\zeta - \phi \in [-\pi, 0]$ and 
$z \cdot w = -\sin(\zeta - \phi) \in [0,1]$. I.e., for $z$ between $z_{1}$ and $-z_{i}$, the sign of $z \cdot w(v)$ is constant in $v$ strictly between $z_{i}$ and $z_{1}$. 
Similarly, for $z$ between $-z_{1}$ and $z_{i}$ the sign of $z \cdot w(v)$ is constant in  in $v$ strictly between $z_{i}$ and $z_{1}$. To sum up, if $B = \varnothing$ then for every $z \in Y(x) \cup (-Y(x))$, the sign of $z \cdot w(v)$ is constant in  in $v$ strictly between $z_{i}$ and $z_{1}$. 

It follows from \eqref{E:angle.deriv.sign} that if $B = \varnothing$ then the derivative 
of $G_{a}(v; x)$ is constant in $v$ between $z_{i}$ and $z_{1}$. But $v=z_{i}$ and $v=z_{1}$ are both minimizers of $G_{a}(v; x)$. Therefore, the derivative of $G_{a}(v; x)$ between $z_{i}$ and $z_{1}$ is 0, i.e., every point between between $z_{i}$ and $z_{1}$ minimizes $G_{a}(\cdot; x)$. This contradicts \eqref{E:Ma(x).in.Y} and the claim that $B \neq \varnothing$ is proved. 
 
Let $z \in B$. So $z \notin \{ z_{1}, z_{i} \}$. Since $z_{i}, z_{j}$ have minimal separation in terms of arc length of any pair of points in $M_{a}(x)$, we have 
$z \notin M_{a}(x)$. There exists $\alpha = 1, \ldots, t$ s.t.\ $z = \pm z_{\alpha}$.  
Let $\phi : (-\pi, \pi] \to S^{1}$ parametrize $S^{1}$ by arc length 
with $\phi(0) = z_{\alpha}$ and $\phi^{-1}(z_{1}) < 0$. So $\phi^{-1}(z_{i}) > 0$. 
Let $\beta = 0, 1, \ldots, t$. If $z_{\beta} = \phi(s)$ with $s \in (-\pi, 0)$ say that $\beta$ and $z_{\beta}$ are ``early''. For example, $\beta = 1$ is early. If $z_{\beta} = \phi(s)$ with $s \in (0,\pi)$ say that $\beta$ is ``late''. For example, $\beta = i$ is late. Let $E(\beta) := 1$ or 0 according as $\beta$ is early or late. 
Thus, if $\beta = 0, \ldots, t$ is neither early nor late, then we must have $z_{\beta} = \pm z_{\alpha}$. If $-z_{\alpha} \in \{ z_{1}, \ldots, z_{t} \}$ 
let $\ell_{-\alpha}$ be the corresponding weight. 
If $-z_{\alpha} \notin \{ z_{1}, \ldots, z_{t} \}$, let $\ell_{-\alpha} := 0$. 
  
For $s \in \RR$ close to 0, define $x(s) \subset S^{1}$ to be the data set with support 
$\{ z_{1}(s), \ldots, z_{t}(s) \}$, where $z_{\beta}(s) = z_{\beta}$
if $\beta \neq \alpha$, and $z_{\alpha}(s) = \phi(s)$. We assume $|s|$ is small enough that 
$\{ z_{1}(s), \ldots, z_{t}(s) \}$ are distinct. Assign to $z_{\beta}(s)$ the same multiplicity 
$\ell_{\beta}$ as before. 

We analyze $M_{a} \bigl( x(s) \bigr)$ for small $|s|$. 
By \eqref{E:Ga.defn}, $G_{a}(v; x)$ is continuous in $x$. Therefore, 
if $z_{\beta} \notin M_{a}(x)$ then for $|s|$ sufficiently small 
$z_{\beta}(s) \notin M_{a} \bigl( x(s) \bigr)$. 

Let $z_{\beta} \in M_{a}(x) \setminus \{ \pm z_{\alpha} \}$. ($z \notin M_{a}(x)$ but $\pm z_{\alpha}$ might equal $-z$ and it is possible that $-z \in M_{a}(x)$.)
Then $z_{\beta}(s) = z_{\beta}$. Notice that 
$\angle \bigl(  z_{\beta} , z_{\alpha}(s) \bigr) = \angle \bigl(  z_{\beta} , z_{\alpha} \bigr) + \bigl( 2 E(\beta) - 1 \bigr) s$. Then
    \begin{align}   \label{E:Ga(z.beta(s),x(s)}
      G_{a} \bigl( z_{\beta}(s); x(s) \bigr) 
        &= G_{a} \bigl( z_{\beta}; x(s) \bigr) \notag \\
        &= a \angle(z_{\beta}, y_{0}) 
          + \ell_{\alpha} \angle( z_{\beta}, z_{\alpha}(s) )  \notag \\
        & \qquad + \ell_{-\alpha} \angle(z_{\beta}, -z_{\alpha}) 
          + \sum_{\gamma \text{ early}} 
            \ell_{\gamma} \angle( z_{\gamma}, z_{\beta} ) 
                + \sum_{\gamma \text{ late}} \ell_{\gamma} 
                  \angle( z_{\gamma}, z_{\beta} ) \notag \\
        &= a \angle(z_{\beta}, y_{0}) + \ell_{\alpha} \angle( z_{\beta}, z_{\alpha} )
           +  \ell_{\alpha} \bigl( 2 E(\beta) - 1 \bigr) s  \\
                   & \qquad + \ell_{-\alpha} \angle(z_{\beta}, -z_{\alpha}) 
          + \sum_{\gamma \text{ early}} 
            \ell_{\gamma} \angle( z_{\gamma}, z_{\beta} ) 
                + \sum_{\gamma \text{ late}} \ell_{\gamma} 
                  \angle( z_{\gamma}, z_{\beta} ) \notag \\
        &= G_{a}(z_{\beta}; x) + \ell_{\alpha} \bigl( 2 E(\beta) - 1 \bigr) s \notag \\
        &= g + \ell_{\alpha} \bigl( 2 E(\beta) - 1 \bigr) s , \notag
    \end{align}
since $z_{\beta} \in M_{a}(x)$. This implies the following. Suppose \emph{$s$ is negative} but close to 0. Then, first, if $\beta$ is early then 
$z_{\beta} \in M_{a} \bigl( x(s) \bigr)$ and 
$G_{a} \bigl( z_{\beta}(s); x(s) \bigr) < g$. And second, if $\beta$ is late then 
$z_{\beta} \notin M_{a} \bigl( x(s) \bigr)$ and 
$G_{a} \bigl( z_{\beta}(s); x(s) \bigr) > g$. 

If \emph{$s$ is positive} but close to 0 we have the opposite: If $\beta$ is late then $z_{\beta} \in M_{a} \bigl( x(s) \bigr)$ and 
$G_{a} \bigl( z_{\beta}(s); x(s) \bigr) < g$. If $\beta$ is early then 
$z_{\beta} \notin M_{a} \bigl( x(s) \bigr)$ and 
$G_{a} \bigl( z_{\beta}(s); x(s) \bigr) > g$. 

Recall that $z = \pm z_{\alpha}$ is a point in $B$. First assume 
    \begin{equation*}
      \text{Neither $z$ nor $-z$ is in } M_{a}(x). 
    \end{equation*}
Then every point of $M_{a}(x)$ is either early or late and we have
    \begin{align} \label{E:precisely.in.Ma(x(s))}
      &\text{If $s$ is negative but close to 0 then the points of  } \notag
        M_{a} \bigl( x(s) \bigr) \text{ are precisely } \\ 
      & \qquad \text{ the early points of $M_{a}(x)$ and } \\
      &\text{If $s$ is positive but close to 0 then the points of  } \notag
        M_{a} \bigl( x(s) \bigr) \text{ are precisely } \\ \notag
      & \qquad \text{ the late points of $M_{a}(x)$. } \notag
    \end{align}

Now assume 
    \begin{equation*}
      z_{\alpha} \in M_{a}(x).
    \end{equation*}
Since $z \notin M_{a}(x)$ we must have $z_{\alpha} = -z$. We show that 
$z_{\alpha}(s) \notin M_{a} \bigl( x(s) \bigr)$ so 
\eqref{E:precisely.in.Ma(x(s))} \emph{continues to hold}.

Consider $G_{a} \bigl( z_{\alpha}(s); x(s) \bigr)$. 
Since $z \notin M_{a}(x)$, we must have $z_{\alpha} = -z \notin B$, 
so $z \neq z_{\alpha}$ and $z_{\alpha} = \phi(\pi)$. 
To compute $G_{a} \bigl( z_{\alpha}(s); x(s) \bigr)$ we need to compute 
$\angle \bigl( y_{0}, z_{\alpha}(s) \bigr)$ for $|s|$ small:
    \begin{equation}  \label{E:angle(y0,z.alpha(s))} 
      \angle \bigl( y_{0}, z_{\alpha}(s) \bigr) =
        \begin{cases}
             \angle ( y_{0}, z_{\alpha} ) + \bigl( 2 E(0) - 1 \bigr) s, 
               &\text{ if } z_{\alpha} \notin \{ \pm y_{0} \}, \\
             |s| = \angle ( y_{0}, z_{\alpha} ) + |s|, &\text{ if } y_{0} = z_{\alpha}, \\
             \pi - |s| = \angle ( y_{0}, z_{\alpha} ) - |s|, &\text{ if } y_{0} = -z_{\alpha} .
        \end{cases}
    \end{equation}
We summarize this by writing 
$\angle \bigl( y_{0}, z_{\alpha}(s) \bigr) = \angle ( y_{0}, z_{\alpha} ) + \Delta_{0} s$. 
So $\Delta_{0} = \Delta_{0}(s) = \Delta_{0}(s;y_{0}) = \pm 1$. Note also that 
$\angle \bigl( z_{\alpha}(s), -z_{\alpha} \bigr) = \angle(z_{\alpha}, -z_{\alpha}) - |s|$. Similarly, 
$\angle \bigl( z_{\alpha}(s), z_{\alpha}(s) ) =  \angle(z_{\alpha}, z_{\alpha}) = 0$. We have then,
    \begin{align}   \label{E:Ga(z.alpha(s),x(s)}
      G_{a} \bigl( z_{\alpha}(s); x(s) \bigr) 
        &= a \angle \bigl( z_{\alpha}(s), y_{0} \bigr) 
          + \ell_{\alpha} \angle \bigl( z_{\alpha}(s), z_{\alpha}(s) \bigr)  
        + \ell_{-\alpha} \angle \bigl( z_{\alpha}(s), -z_{\alpha} \bigr) \notag \\
          & \qquad + \sum_{\gamma \text{ early}} \ell_{\gamma} 
            \angle \bigl(  z_{\gamma}, z_{\alpha}(s) \bigr) 
                + \sum_{\gamma \text{ late}} \ell_{\gamma} 
                  \angle \bigl( z_{\gamma}, z_{\alpha}(s) \bigr) \notag \\
        &= a \angle(z_{\alpha}, y_{0}) + a \Delta_{0}(s) \, s 
          + \ell_{\alpha} \angle( z_{\alpha}, z_{\alpha} )
        + \ell_{-\alpha} \angle(z_{\alpha}, -z_{\alpha}) - \ell_{-\alpha} |s|  \\
          & \qquad + \sum_{\gamma \text{ early}} \ell_{\gamma} 
            \bigl[ \angle( z_{\gamma}, z_{\alpha} ) + s \bigr]
                + \sum_{\gamma \text{ late}} \ell_{\gamma} 
                  \bigl[ \angle( z_{\gamma}, z_{\alpha} ) - s \bigr] \notag \\
        &= G_{a}(z_{\alpha}; x) 
          + \left( a \Delta_{0}(s) - sign(s) \, \ell_{-\alpha} 
            + \sum_{\gamma \text{ early}} \ell_{\gamma} 
              - \sum_{\gamma \text{ late}} \ell_{\gamma} \right) \notag s \\
        &= g + \left( a \Delta_{0}(s) - sign(s) \, \ell_{-\alpha} 
            + \sum_{\gamma \text{ early}} \ell_{\gamma} 
              - \sum_{\gamma \text{ late}} \ell_{\gamma} \right) s \notag .
    \end{align}
since $z_{\alpha} \in M_{a}(x)$, by assumption.

Suppose
     \begin{equation*}
       z_{\alpha} = - y_{0} \in M_{a}(x) .
    \end{equation*}
Choose $w \in S^{1}$ s.t.\ $w \perp z_{\alpha}$ and the sign of $z_{\gamma} \cdot w$ is positive for late $\gamma$ and negative for early. Since $z_{\alpha} \in M_{a}(x)$, we must have 
$\tfrac{d_{-}}{dv} G_{a}(v;x) \restriction_{v=z_{\alpha}} < 0$ and 
$\tfrac{d_{+}}{dv} G_{a}(v;x)  \restriction_{v=z_{\alpha}} > 0$. Thus, applying \eqref{E:angle.deriv.sign} and \eqref{E:one-sided.angle.derivs} then moving the one-sided derivative at $z_{\alpha}$ over to the RHS we get:
    \begin{align}  \label{E:LEFT.RIGHT.ineqs.z.alpha.=.-y0}
       \emph{LEFT: } &a + \ell_{-\alpha} 
         + \sum_{\gamma \text{ early}} \ell_{\gamma} 
                  - \sum_{\gamma \text{ late}} \ell_{\gamma} < \ell_{\alpha} , \\
       \emph{RIGHT: }  -&a - \ell_{-\alpha} 
         + \sum_{\gamma \text{ early}} \ell_{\gamma} 
                  - \sum_{\gamma \text{ late}} \ell_{\gamma} > - \ell_{\alpha} \notag .
    \end{align}

By \eqref{E:angle(y0,z.alpha(s))}, $\Delta_{0}(s)= - sign(s)$. Therefore, the LHS of the first of these inequalities is the coefficient 
of $s$ in \eqref{E:Ga(z.alpha(s),x(s)} when $s < 0$. The RHS the first inequality is the coefficient of $s$ in \eqref{E:Ga(z.beta(s),x(s)} corresponding to early $\beta$. Multiplying both sides 
by $s < 0$ we get, by \eqref{E:Ga(z.alpha(s),x(s)} and \eqref{E:Ga(z.beta(s),x(s)},
    \begin{equation*}
      G_{a} \bigl( z_{\alpha}(s); x(s) \bigr) 
        = g + \left( a + \ell_{-\alpha} + \sum_{\gamma \text{ early}} \ell_{\gamma} 
          - \sum_{\gamma \text{ late}} \ell_{\gamma} \right) s
            > g + \ell_{\alpha} s = G_{a} \bigl( z_{\beta}(s); x(s) \bigr),
    \end{equation*}
providing $\beta$ is early and $s < 0$ is close to 0. 
Thus, $z_{\alpha}(s) \notin M_{a} \bigl( x(s) \bigr)$

The LHS of the second inequality in \eqref{E:LEFT.RIGHT.ineqs.z.alpha.=.-y0}, is the coefficient of $s$ in \eqref{E:Ga(z.alpha(s),x(s)} when $s > 0$. The RHS the second inequality is the coefficient of $s$ in \eqref{E:Ga(z.beta(s),x(s)} corresponding to late $\beta$. Proceeding as in the last paragraph we again find $z_{\alpha}(s) \notin M_{a} \bigl( x(s) \bigr)$.

Similarly, suppose 
    \begin{equation*}
       z_{\alpha} = y_{0} \in M_{a}(x) .
    \end{equation*} 
Let $w \in S^{1}$ be as before. Since $z_{\alpha} \in M_{a}(x)$, we must have $\tfrac{d_{-}}{dv} \restriction_{v=z_{\alpha}} < 0$ and 
$\tfrac{d_{+}}{dv} \restriction_{v=z_{\alpha}} > 0$. Thus, applying \eqref{E:angle.deriv.sign} and \eqref{E:one-sided.angle.derivs} we get:
    \begin{align*}
       \emph{LEFT: } -&a + \ell_{-\alpha} 
         + \sum_{\gamma \text{ early}} \ell_{\gamma} 
                  - \sum_{\gamma \text{ late}} \ell_{\gamma} < \ell_{\alpha} , \\
       \emph{RIGHT: }  &a - \ell_{-\alpha} 
         + \sum_{\gamma \text{ early}} \ell_{\gamma} 
                  - \sum_{\gamma \text{ late}} \ell_{\gamma} > - \ell_{\alpha} .
    \end{align*}

By \eqref{E:angle(y0,z.alpha(s))}, $\Delta_{0}(s)= sign(s)$. Therefore, the LHS of the first of these inequalities is the coefficient of $s$ in \eqref{E:Ga(z.alpha(s),x(s)} when $s < 0$. The RHS of the first inequality is the coefficient of $s$ in \eqref{E:Ga(z.beta(s),x(s)} corresponding to early $\beta$. 
The LHS of the second inequality is the coefficient of $s$ in \eqref{E:Ga(z.alpha(s),x(s)} when $s > 0$. The RHS the second inequality is the coefficient of $s$ in \eqref{E:Ga(z.beta(s),x(s)} corresponding to late $\beta$. Arguing as before we again conclude that 
$z_{\alpha}(s) \notin M_{a} \bigl( x(s) \bigr)$.

Finally, consider the case 
    \begin{equation*}
      z_{\alpha} \in M_{a}(x) \setminus \{ \pm y_{0} \} .
    \end{equation*} 
Then, from \eqref{E:angle(y0,z.alpha(s))}, we have 
$\Delta_{0}(s) = \bigl( 2 E(0) - 1 \bigr)$. Hence, the usual argument leads to the inequalities
    \begin{align*}
       \emph{LEFT: } &a \bigl( 2 E(0) - 1 \bigr) + \ell_{-\alpha} 
         + \sum_{\gamma \text{ early}} \ell_{\gamma} 
           - \sum_{\gamma \text{ late}} \ell_{\gamma} < \ell_{\alpha} , \\
       \emph{RIGHT: }  &a \bigl( 2 E(0) - 1 \bigr) - \ell_{-\alpha} 
          + \sum_{\gamma \text{ early}} \ell_{\gamma} 
            - \sum_{\gamma \text{ late}} \ell_{\gamma} > - \ell_{\alpha} .
    \end{align*}
By the usual argument and the results in the previous two cases, we conclude that when $z_{\alpha} \notin \{ \pm y_{0} \}$ and $s \neq 0$ but close to 0 then $z_{\alpha}(s) \notin M_{a} \bigl( x(s) \bigr)$. 
 
We conclude that by perturbing $z_{\alpha}$ by $s$ close to zero we get $z_{\alpha}(s) \notin M_{a} \bigl( x(s) \bigr)$. Hence, \eqref{E:precisely.in.Ma(x(s))} holds in general.

We observed above that $\beta = 1$ is early and $\beta = i$ is late. Hence, after the perturbation, exactly one of $z_{1}$ and $z_{i}$ is in $M_{a} \bigl( x(s) \bigr)$. In particular, $0 < \bigl| M_{a} \bigl( x(s) \bigr) \bigr| < \bigl| M_{a}(x) \bigr|$, 
where $|S|$ is the cardinality of a set $S$. But the important point is that with $|s|$ small, we get data sets $x_{-}(s)$ and $x_{+}(s)$ corresponding to $s<0$ and $s>0$, resp., s.t.\ $M_{a} \bigl( x_{+}(s) \bigr)$ and $M_{a} \bigl( x_{-}(s) \bigr)$ 
are nonempty disjoint subsets of, resp., the late and early sets of points in $x$. 
Importantly, $x_{\pm}(s) \to x$ as $s \to 0$. 

Now, apply this procedure recursively to each of $x_{-}(s)$ and $x_{+}(s)$. We eventually end up with at least two data sets $x_{minus}(\mbf{s})$ and 
$x_{plus}(\mbf{s}')$ depending on vectors $\mbf{s}$ and $\mbf{s}'$ close to 0 and having the following properties. First, $M_{a}\bigl( x_{minus}(\mbf{s}) \bigr)$ contains exactly one point in $S^{1}$, one of the early points in $M_{a}(x)$ and $M_{a}\bigl( x_{plus}(\mbf{s}) \bigr)$ contains exactly one point, one of the late points in $M_{a}(x)$. This means 
$x_{minus}(\mbf{s}), x_{plus}(\mbf{s}') \in \D'$. Second, 
as $\mbf{s}, \mbf{s}', s \to 0$ we have 
$x_{minus}(\mbf{s}), x_{plus}(\mbf{s}') \to x$, our original data set 
in $\D \setminus \D'$. This means $x$ is in the closure of $\D'$. 
Since $x \in \D \setminus \D'$ is arbitrary this shows $\D'$ is dense in $\D$. 

The singleton $M_{a}\bigl( x_{minus}(\mbf{s}) \bigr)$ is among the early points of $M_{a}(x)$ and $M_{a}\bigl( x_{plus}(\mbf{s}) \bigr)$ is among the late. The early and late sets are a positive distance from each other. It follows that $x$ is a singularity (w.r.t.\ $\D'$). 
  \end{proof}

  \begin{proof}[Proof of proposition \ref{P:severe.sings.of.ma.dist}] 
Let $k=1$. Let $c \in [n-3, n-2)$ be a fixed. Its value will be chosen later. First, suppose $a \in (0, c]$ not be an integer. Then
    \begin{equation*}
      \frac{dist_{n-2} \bigl(\Ss_{a}^{\msf{V}_{\pi/2}}, \Pf_{k} \bigr)}{n-2-a}
        \leq \frac{diam(\D)}{n-2-c} < \infty .
    \end{equation*}
Thus, 
    \begin{equation}  \label{E:a.in.(0, c]}
      \text{\eqref{E:dist.to.Pk.O(n-2-a)} holds for $a \in (0, c]$, non-integer.}
    \end{equation}

Recall \eqref{E:ma.n.k.a.sizes}. Suppose
    \begin{equation} \label{E:a.tween.c.and.n-2}
      n-3 \leq c < a < n-2 .
    \end{equation}
Let  
    \begin{equation*}
      y_{0} := (-1, 0).
    \end{equation*}
Specify positions on $S^{1}$ by signed angles from $-y_{0} = (1,0)$, with counterclockwise the positive direction, as usual. Let $\ell$ be the integer satisfying 
    \begin{equation}   \label{E:0.leq.n-2ell.leq.1}
      0 \leq n-2\ell \leq 1 .
    \end{equation}
Since $n \geq 4$ by assumption, $\ell \geq 2$. 

Let 
    \begin{equation}  \label{E:theta.order}
      -\pi/2 < \theta_{1} < \cdots < \theta_{\ell} < 0 < \theta_{\ell+1} 
        < \cdots < \theta_{n-1} < \pi/2 .
    \end{equation}
Later we will let all of them go to 0. Let 
    \begin{equation}  \label{E:alpha.<.-theta.ell}
       \alpha \in (0, -\theta_{\ell}) .
    \end{equation}  
It is convenient to use complex numbers to represent points on the plane. We use $\sqrt{-1}$ instead of $i$ because we use $i$ as an index.
Let $y_{i} := \exp( \theta_{i} \sqrt{-1}) \in S^{1}$ ($i \in \NN_{n-1}$) and let 
    \begin{equation*}
       y_{n} := \exp\bigl[ (\pi-\alpha) \sqrt{-1} \, \bigr] .
    \end{equation*}
Since the $\theta_{i}$'s are within $\pi/2$ of 0, we have, e.g., $\angle(y_{1}, y_{n-1}) = \theta_{n-1} - \theta_{1}$. We also have 
$\angle(y_{n}, y_{i}) = \pi + \alpha + \theta_{i}$ if $i \leq \ell$. 
Let $x := (y_{1}, \ldots, y_{n}) \in \D$. We derive conditions on the $\theta_{i}$'s 
and $\alpha$ s.t.\ $M_{a}(x) = \{ y_{0}, y_{1}, y_{n-1} \}$ and, for $a$ and a $c$ satisfying \eqref{E:a.tween.c.and.n-2}, exhibit $x$ satisfying these conditions. In fact, we will see that there is a subset of of $\D$ of positive $\Hm^{n-2} = \Hm^{d-\nvar-1}$-measure (see \eqref{E:ma:nvar=1}) for which those conditions are satisfied. 
Recall that $y_{0} := (-1, 0) = \exp( \pi \sqrt{-1} )$. 
Thus, $\angle(y_{0}, y_{n}) = \alpha$. We have
    \begin{align} \label{E:3Ga(yi)s}
      G_{a}(y_{0}) &= \sum_{j=1}^{\ell-1} (\pi+\theta_{j}) + (\pi + \theta_{\ell}) 
        + \sum_{j=\ell+1}^{n-1} (\pi-\theta_{j}) + \alpha, \notag \\
        &= (n-1) \pi + \sum_{j=1}^{\ell} \theta_{j} - \sum_{j=\ell+1}^{n-1} \theta_{j} 
          + \alpha, \notag \\
      G_{a}(y_{1}) &= a(\pi+\theta_{1}) + \sum_{j=1}^{\ell-1} (\theta_{j}-\theta_{1})
        + (\theta_{\ell} - \theta_{1}) + \sum_{j=\ell+1}^{n-1} (\theta_{j}-\theta_{1})
          + (\pi + \theta_{1} + \alpha) \notag \\
                          &= (a+1) \pi + (a-n+2) \theta_{1} + \sum_{j=1}^{\ell} \theta_{j} 
                            + \sum_{j=\ell+1}^{n-1} \theta_{j} + \alpha , \\
      G_{a}(y_{n-1}) &= a(\pi-\theta_{n-1}) + \sum_{j=1}^{\ell-1} (\theta_{n-1}-\theta_{j})
        + (\theta_{n-1}-\theta_{\ell}) + \sum_{j=\ell+1}^{n-1} (\theta_{n-1}-\theta_{j})  \notag \\
                              & \qquad + (\pi - \theta_{n-1} - \alpha)   \notag \\
                              &= (a+1) \pi + (-a + n - 2) \theta_{n-1}
                                - \sum_{j=1}^{\ell} \theta_{j} 
                                  - \sum_{j=\ell+1}^{n-1} \theta_{j} - \alpha  \notag .
    \end{align}

In order that $M_{a}(x) = \{ y_{0}, y_{1}, y_{n-1} \}$, we need $G_{a}(y_{1}) = G_{a}(y_{0})$. By \eqref{E:3Ga(yi)s}, we have, 
    \begin{align} \label{E:Ga.y1.-.Ga.y0}
      G_{a}(y_{1}) - G_{a}(y_{0}) &= \bigl( a+1 - (n-1) \bigr) \pi + (a-n+2) \theta_{1} \notag \\
        & \qquad + \sum_{j=1}^{\ell} (\theta_{j} - \theta_{j}) 
          + 2 \sum_{j=\ell+1}^{n-1} \theta_{j} + (\alpha - \alpha) \\
        &= - (n-2-a) \pi - (n-2-a) \theta_{1}  + 2 \theta_{n-1} + 2 \sum_{j=\ell+1}^{n-2} \theta_{j} 
          \notag .
    \end{align}
Thus, we arrive at the equation
    \begin{equation} \label{E:Ga.y1.-.Ga.y0.eqn}
      - (n-2-a) \theta_{1} + 2 \theta_{n-1} = - 2 \sum_{j=\ell+1}^{n-2} \theta_{j} + (n-2-a) \pi .
    \end{equation}

But we also need $G_{a}(y_{n-1}) = G_{a}(y_{0})$. By \eqref{E:3Ga(yi)s},
    \begin{align} \label{E:Ga.y.n-1.-.Ga.y0}
      G_{a}(y_{n-1}) - G_{a}(y_{0}) &= \bigl( a+1 - (n-1) \bigr) \pi + (- a + n-2) \theta_{n-1} \notag \\
        & \qquad \qquad - 2 \sum_{j=1}^{\ell} \theta_{j} - 2 \alpha \\
          &= - (n-2-a) \pi + (n-2-a) \theta_{n-1} - 2 \theta_{1} 
            - 2 \sum_{j=2}^{\ell} \theta_{j} - 2 \alpha \notag .
    \end{align}
Thus, we arrive at the equation.
    \begin{equation} \label{E:Ga.y.n-1.-.Ga.y0.eqn}
      - 2 \theta_{1} + (n-2-a) \theta_{n-1} = 2 \sum_{j=2}^{\ell} \theta_{j} + 2 \alpha 
        + (n-2-a) \pi .
    \end{equation}

Denote the desired common value of $G_{a}(y_{0})$, $G_{a}(y_{1})$, and $G_{a}(y_{n-1})$ by
    \begin{equation*}
       g := G_{a}(y_{0}) = G_{a}(y_{1}) = G_{a}(y_{n-1}) .
    \end{equation*}
In order that $M_{a}(x) = \{ y_{0}, y_{1}, y_{n-1} \}$, 
we also need $G_{a}(y_{i}) > g$ for $i = 2, \ldots, n-2, n$. 
First, suppose $i = 2, \ldots, \ell$. 
$n \geq 4$, so $n-2 \geq 2$. By \eqref{E:theta.order}, this means $\theta_{i} < 0$. As observed above, $\angle(y_{n}, y_{i}) = \pi + \alpha + \theta_{i}$. Thus, we have
    \begin{equation*}
          G_{a}(y_{i}) = a(\pi+\theta_{i}) + \sum_{j=1}^{i-1} (\theta_{i}-\theta_{j})
             + \sum_{j=i}^{\ell} (\theta_{j}-\theta_{i}) + \sum_{j=\ell+1}^{n-1} (\theta_{j}-\theta_{i})   
                 + \pi + \theta_{i} + \alpha . 
    \end{equation*}
Then from \eqref{E:3Ga(yi)s}, we have 
    \begin{align} \label{E:Ga.yi.-.Ga.y1.i.leq.ell}
      G_{a}(y_{i}) - g &= G_{a}(y_{i}) - G_{a}(y_{1})  \notag \\
        &= \left[ (a+1) \pi + (a+2i-n) \theta_{i} - \sum_{j=1}^{i-1} \theta_{j} 
          + \sum_{j=i}^{\ell} \theta_{j} + \sum_{j=\ell+1}^{n-1} \theta_{j} + \alpha \right] \notag \\
        & \qquad - \left[ (a+1) \pi + (a-n+2) \theta_{1} + \sum_{j=1}^{i-1} \theta_{j}
               + \sum_{j=i}^{\ell} \theta_{j} + \sum_{j=\ell+1}^{n-1} \theta_{j} + \alpha \right] \\
        &= (a+2i-n) \theta_{i} - (a-n+2) \theta_{1} - 2 \sum_{j=1}^{i-1} \theta_{j}   \notag \\
        &= \bigl[ 2(i-1) - (n-2-a) \bigr]  \theta_{i} + (n-2-a) \theta_{1} 
           - 2 \sum_{j=1}^{i-1} \theta_{j} .  \notag
    \end{align}
Thus, we have the inequalities,
    \begin{equation} \label{E:theta.a.ineq.small.alpha.i.leq.ell}
      2 \sum_{j=1}^{i-1} (\theta_{i}-\theta_{j}) > (n-2-a) ( \theta_{i} - \theta_{1} ) , 
        \quad i = 2, \ldots, \ell, 
    \end{equation}
Recall $0 < n-2-a < n-2-c \leq 1$, by \eqref{E:a.tween.c.and.n-2}. So \eqref{E:theta.a.ineq.small.alpha.i.leq.ell} holds for $i=2$. Inductively, suppose \eqref{E:theta.a.ineq.small.alpha.i.leq.ell} holds for $i = h$. Then
    \begin{equation*}
      2 \sum_{j=1}^{(h+1)-1} (\theta_{i}-\theta_{j}) 
        > (n-2-a) \bigl[ ( \theta_{h} - \theta_{1} ) + ( \theta_{h+1} - \theta_{h} ) \bigr]
          = (n-2-a) ( \theta_{h+1} - \theta_{1} ) .
    \end{equation*}
Thus, $G_{a}(y_{i}) > g$ \emph{always holds} for $i = 2, \ldots, \ell$. 

Now take $i = \ell+1, \ldots, n-2$ (possible because $n \geq 4$). By \eqref{E:theta.order}, this means $\theta_{i} > 0$ so $\angle(y_{n}, \theta_{i}) = \pi - \alpha - \theta_{i}$. We have 
    \begin{align*}
      G_{a}(y_{i}) &= a(\pi-\theta_{i}) + \sum_{j=1}^{\ell} (\theta_{i} - \theta_{j}) 
        + \sum_{j=\ell+1}^{i-1} (\theta_{i} - \theta_{j}) + \sum_{j=i}^{n-1} (\theta_{j} - \theta_{i})
          + \pi - \alpha - \theta_{i} \\
             &= (a+1) \pi + (-a + 2i - n - 2) \theta_{i} 
               - \sum_{j=1}^{\ell} \theta_{j} - \sum_{j=\ell+1}^{i-1} \theta_{j} 
                 + \sum_{j=i}^{n-1} \theta_{j} - \alpha .
    \end{align*}
Hence, by \eqref{E:3Ga(yi)s},
    \begin{align*}
      G_{a}(y_{i}) - G_{a}(y_{n-1}) 
        &= \Biggl[ (a+1) \pi + (-a + 2i - n - 2) \theta_{i} 
               - \sum_{j=1}^{\ell} \theta_{j} - \sum_{j=\ell+1}^{i-1} \theta_{j} 
                 + \sum_{j=i}^{n-1} \theta_{j} - \alpha \Biggr] \\
         & \qquad - \Biggl[ (a+1) \pi + (n - 2 - a) \theta_{n-1}
                - \sum_{j=1}^{\ell} \theta_{j} - \sum_{j=\ell+1}^{i-1} \theta_{j} 
                 - \sum_{j=i}^{n-1} \theta_{j} - \alpha \Biggr] \\
         &= (-a + 2i - n - 2) \theta_{i} - (n - 2 - a) \theta_{n-1} 
           + 2 \sum_{j=i}^{n-1} \theta_{j} 
         &\intertext{  \qquad \qquad \qquad Moving the $\theta_{i}$ 
           and $\theta_{n-1}$ out from behind the $\sum$: } 
         &= (-a + 2i - n) \theta_{i} + 2 \theta_{n-1} - (n - 2 - a) \theta_{n-1} + 2 \sum_{j=i+1}^{n-2} \theta_{j} 
    \end{align*}
This yields a third set of inequalities,
    \begin{equation} \label{E:theta.a.ineq.i>ell}
      (-a + 2i - n) \theta_{i} + \theta_{n-1} + \bigl[ 1 - (n - 2 - a) \bigr] \theta_{n-1} 
        + 2 \sum_{j=i+1}^{n-2} \theta_{j} > 0,
        \quad i = \ell+1, \ldots, n-2 .
    \end{equation}
The LHS of the preceding can be rewritten
    \begin{equation*} 
      (2i - n) \theta_{i} + \theta_{n-1} - (n-3) \theta_{n-1} + (\theta_{n-1} - \theta_{i}) a
        + 2 \sum_{j=i+1}^{n-2} \theta_{j} .
    \end{equation*}
By \eqref{E:a.tween.c.and.n-2}, $a > c \geq n-3$. By \eqref{E:theta.order}, 
$\theta_{n-1} - \theta_{i} > 0$. Therefore, a lower bound to the preceding is obtained by replacing $a$ by $n-3$: 
    \begin{multline*}  
      (2i - n) \theta_{i} + \theta_{n-1} - (n-3) \theta_{n-1} + (\theta_{n-1} - \theta_{i})(n-3) 
        + 2 \sum_{j=i+1}^{n-2} \theta_{j} \\
          > (2i - n) \theta_{i} + \theta_{n-1} - (n-3) \theta_{i} 
            + 2 \sum_{j=i+1}^{n-2} \theta_{j} .
    \end{multline*}
But by \eqref{E:theta.order} $\theta_{j} > \theta_{i}$ for $j = i+1, \ldots, n-2$. Therefore, the followng is a lower bound to the LHS of \eqref{E:theta.a.ineq.i>ell}. 
    \begin{multline*} 
      (2i - n) \theta_{i} + \theta_{n-1} - (n-3) \theta_{i} 
        + 2 (n-2-i) \theta_{i} \\
          > \theta_{n-1} + \bigl[ (2i - n) - (n-3) + 2 (n-2-i) \bigr] \theta_{i} 
            = \theta_{n-1} - \theta_{i} > 0 .
    \end{multline*}
Thus, $G_{a}(y_{i}) > g$  
\emph{always holds} for $i = \ell+1, \ldots, n-2$.

Finally, we prove $G_{a}(y_{n}) > g$. 
    \begin{align*}
      G_{a}(y_{n}) &= a \alpha + \sum_{i=1}^{\ell} (\pi + \theta_{j} + \alpha)
        + \sum_{i=\ell+1}^{n-1} (\pi - \theta_{j} - \alpha) \\
          &= (n-1) \pi + \sum_{i=1}^{\ell} \theta_{j} 
            - \sum_{i=\ell+1}^{n-1} \theta_{j} + (a + 2 \ell - n + 1) \alpha .
    \end{align*}
Then, by \eqref{E:3Ga(yi)s},
    \begin{align*}
      G_{a}(y_{n}) - &G_{a}(y_{0 }) \\
        &=  \left[ (n-1) \pi + \sum_{i=1}^{\ell} \theta_{j} 
                - \sum_{i=\ell+1}^{n-1} \theta_{j} + (a + 2 \ell - n + 1) \alpha \right] \\
        & \qquad - \left[ (n-1) \pi + \sum_{j=1}^{\ell} \theta_{j} - \sum_{j=\ell+1}^{n-1} \theta_{j} 
              + \alpha \right] \\
        &= (a + 2 \ell - n) \alpha
    \end{align*}
Thus, by \eqref{E:a.tween.c.and.n-2} and \eqref{E:0.leq.n-2ell.leq.1}, 
$G_{a}(y_{n}) - G_{a}(y_{0 }) > (n-3)-1$. But $n \geq 4$ by hypothesis. Therefore, 
$G_{a}(y_{n}) > g$ always holds. \emph{This completes the proof that}
\begin{equation}  \label{E:Ga(yi)>g}
  G_{a}(y_{i}) > g \text{ for } i = 2, \ldots, n-2, n .
\end{equation}

Of course, we also have the inequalities \eqref{E:theta.order} and \eqref{E:alpha.<.-theta.ell}. In these and the equations \eqref{E:Ga.y1.-.Ga.y0.eqn} and \eqref{E:Ga.y.n-1.-.Ga.y0.eqn}, $a$ is given, satisfying \eqref{E:a.tween.c.and.n-2}.

By hypothesis, $n > 3$ so $\ell \geq 2$. To show there are solutions to equations \eqref{E:Ga.y1.-.Ga.y0.eqn} and \eqref{E:Ga.y.n-1.-.Ga.y0.eqn} we parametrize 
the $\theta_{i}$'s. 
Define and choose
    \begin{equation}  \label{E:thetas.zeta.omega}
        \theta_{i} := 
             \begin{cases}
               (i-\ell) \zeta,                   &\text{if } i = 1, \ldots, \ell-1, \\
               (i-\ell) \omega,               &\text{if } i = \ell+1, \ldots, n-1 .
             \end{cases}
    \end{equation}
We will arrange things so that $\zeta \in (0, \tfrac{\pi}{2(\ell-1)})$ 
and $\omega \in (0, \tfrac{\pi}{2(n-\ell-1})$. We will then choose
    \begin{equation} \label{E:theta.ell.alpha}
      \theta_{\ell} \in (-\zeta, 0) \text{ and } \alpha \in (0, -\theta_{\ell}) .
    \end{equation}
These inequalities together with $\zeta \in (0, \tfrac{\pi}{2(\ell-1)})$ 
and $\omega \in (0, \tfrac{\pi}{2(n-\ell-1})$ replace \eqref{E:theta.order} and \eqref{E:alpha.<.-theta.ell}.

We have already observed that \eqref{E:Ga(yi)>g} is automatically satisfied. 
Thus, we need only find $\zeta, \omega > 0$ and appropriate $\theta_{\ell}$ and $\alpha$ so that \eqref{E:Ga.y1.-.Ga.y0.eqn} and \eqref{E:Ga.y.n-1.-.Ga.y0.eqn} hold. 
Let $A := n-2-a$, $B := n-\ell-1$, $C := \ell-1$, and $D := n - \ell$. Then, by \eqref{E:a.tween.c.and.n-2}, \eqref{E:0.leq.n-2ell.leq.1}, and the fact that $n \geq 4$, we have 
    \begin{equation*}
      0 < A < n-2-c \leq 1, \; B \geq 1, \; C \geq 1, \text{ and } D \geq 2.
    \end{equation*}  
Let $E := A \pi$ and $F := E + 2 \alpha + 2 \theta_{\ell} < E$, if \eqref{E:theta.ell.alpha} holds. Then based on \eqref{E:thetas.zeta.omega}, the equations \eqref{E:Ga.y1.-.Ga.y0.eqn} and \eqref{E:Ga.y.n-1.-.Ga.y0.eqn} become
    \begin{align*} \label{E:BIG.LETTER.eqn.system}
          AC \zeta +B D \omega &= E, \\
      C \ell \zeta + A B \omega &= F .
    \end{align*}
It is easy to see that the unique solutions to this system are
     \begin{subequations} \label{E:zeta.omega.solns}
         \begin{align*}  
                 \zeta &= \frac{1}{A C} \left( E - \frac{D}{D \ell - A^{2}}(E \ell - A F) \right)
                = \frac{D F - AE}{C(D \ell - A^{2})} , \tag{\ref{E:zeta.omega.solns}a} \\
                         &= \frac{(D - A)E - 2D(- \theta_{\ell} - \alpha)}
                           {C(D \ell - A^{2})} , \notag \\
                 \omega &= \frac{E \ell - AF}{B (D \ell - A^{2})}
                    = \frac{E(\ell - A) - 2 A(\alpha+\theta_{\ell})}{B (D \ell - A^{2}) }
                      = A \frac{\pi (\ell - A) + 2 (- \theta_{\ell} - \alpha)}{B (D \ell - A^{2})} 
                        \tag{\ref{E:zeta.omega.solns}b}.
        \end{align*}
    \end{subequations}
Observe that $\ell - A$, $D - A$, $C(D \ell - A^{2})$, and $B (D \ell - A^{2})$ 
are all $\geq 1$. Hence, $\omega > 0$, providing \eqref{E:theta.ell.alpha} holds. We will show that things can be arranged so that $\zeta > 0$, as well.  
It is immediate that $\omega = O(A)$ as $A = n-2-a \downarrow 0$, 
i.e. as $a \uparrow n-2$. We will see presently that the same is true of $\zeta$.

 In accordance with \eqref{E:thetas.zeta.omega}, we want $\theta_{\ell} \in (-\zeta, 0)$. Recalling that $A \in (0,1)$ and $\alpha > 0$, by \eqref{E:theta.ell.alpha},  
and $E := A \pi$, we see that it suffices to take $\theta_{\ell} < 0$ s.t.\ 
    \begin{multline*}
      C(D \ell - 1) \theta_{\ell} > C(D \ell - A^{2}) \theta_{\ell} 
        > - (D - 1) A \pi - 2 D  \theta_{\ell} \\
          > - (D - A) A \pi - 2 D  \theta_{\ell} - 2 D \alpha
            = - C(D \ell - A^{2}) \zeta .
    \end{multline*}
Note that, since $\ell \geq 2$, $C(D \ell - 1) + 2 D \geq 7 > 0$. Therefore, it suffices to choose $\theta_{\ell}$ satisfying 
    \begin{equation}  \label{E:theta.ell.ACD.lwr.bound}
      0 > \theta_{\ell} > - \frac{(D - 1) A \pi}{C(D \ell - 1) + 2 D} .
    \end{equation}
This requires $\theta_{\ell} = O(n-2-a)$ as $a \uparrow n-2$. 

Choose $\alpha \in - \theta_{\ell} (1/2, 1)$, so $\alpha_{\ell} = O(n-2-a)$. Then $\theta_{\ell} + \alpha > \theta_{\ell}/2$. Then, by the preceding, 
    \begin{multline*}
      (D - A)E - 2D(- \theta_{\ell} - \alpha) > (D - A)E + D \theta_{\ell}
        > (D - A)E - \frac{D(D - 1) A \pi}{C(D \ell - 1) + 2 D} \\
          > (D - A) A \pi - \frac{D(D - A) A \pi}{C(D \ell - 1) + 2 D} .
    \end{multline*}
The final expression is strictly positive if and only if 
    \begin{equation*}
      \bigl[ C(D \ell - 1) + 2 D \bigr] - D = C(D \ell - 1) + D > 0 .
    \end{equation*}
But $C \geq 1$ and $D$ and $\ell$ are both at least 2, so the preceding is true. Thus, if \eqref{E:theta.ell.ACD.lwr.bound} and 
$\alpha \in - \theta_{\ell} (1/2, 1)$ hold then \eqref{E:theta.ell.alpha} holds; $\zeta$, defined in (\ref{E:zeta.omega.solns}a), is positive; and $\zeta = O(n-2-a)$ as $a \uparrow n-2$. In particular, for $c$ in \eqref{E:a.tween.c.and.n-2}, and hence $a$, sufficiently close to $n-2$ and the $\theta_{i}$'s defined by \eqref{E:thetas.zeta.omega}, we have that \eqref{E:theta.order} holds.

Let $\theta_{\ell} = \theta_{0 \ell}$ satisfy \eqref{E:theta.ell.ACD.lwr.bound} 
and $\alpha = \alpha_{0} \in - \theta_{0 \ell} (1/2, 1)$. ($\theta_{0 \ell}$ 
and $\alpha_{0}$ depend on $\alpha$.) Define $\zeta$ and $\omega$ by \eqref{E:zeta.omega.solns}, and $\theta_{1}, \ldots, \theta_{n-1}$ by \eqref{E:thetas.zeta.omega}. Write 
$x_{a} := x_{a}(\theta_{0 \ell}, \alpha_{0}) := \bigl( y_{1}(\theta_{0 \ell}, \alpha_{0}), 
\ldots, y_{n}(\theta_{0 \ell}, \alpha_{0}) \bigr) \in \D$ 
with $y_{i}(\theta_{\ell}, \alpha) := \exp( \theta_{i} \sqrt{-1} \, ) \in S^{1}$ 
($i \in \NN_{n-1}$) and $y_{n} := \exp \bigl[ (\pi-\alpha) \sqrt{-1} \, \bigr]$. Then 
$M_{a} \bigl( x_{a}(\theta_{\ell}, \alpha) \bigr) = \bigl\{ y_{0}(\theta_{\ell}, \alpha), 
y_{1}(\theta_{\ell}, \alpha), y_{n-1}(\theta_{\ell}, \alpha) \bigr\}$.
\emph{Claim:} For each $j \in \{ 0, 1, n-1 \}$, there is a sequence 
$\{ x_{\gamma}, \gamma = 1, 2, \ldots \} \subset \D'$ s.t.\ 
$x_{\gamma} \to x_{a}(\theta_{0 \ell}, \alpha_{0})$, and $m_{a}(x_{\gamma})$ 
(defined since $x_{\gamma} \in \D'$) converges to $y_{j}(\theta_{0 \ell}, \alpha_{0})$.

To prove the claim we show how to construct a $x_{\gamma}$. 
Let $\theta_{1}, \ldots, \theta_{n}$ satisfying \eqref{E:theta.order} and $\alpha$ satisfying \eqref{E:alpha.<.-theta.ell}, but we drop the requirements \eqref{E:thetas.zeta.omega} and \eqref{E:theta.ell.alpha}. Let $y_{i} := \exp( \theta_{i} \sqrt{-1} \, ) \in S^{1}$ 
($i \in \NN_{n-1}$) and $y_{n} := \exp \bigl[ (\pi-\alpha) \sqrt{-1} \, \bigr]$. 
Let $x := (y_{1}, \ldots, y_{n})$. Note that by \eqref{E:Ga.y1.-.Ga.y0} and \eqref{E:Ga.y.n-1.-.Ga.y0},
    \begin{equation}  \label{E:matrix.version.of.eqns}
             \begin{pmatrix}
               G_{a}(y_{1};x) - G_{a}(y_{0};x) \\
               G_{a}(y_{n-1};x) - G_{a}(y_{0};x)
             \end{pmatrix}
         =
             \begin{pmatrix}
               -A & 2 \\
               -2 & A
             \end{pmatrix} 
             \begin{pmatrix}
               \theta_{1} \\
               \theta_{n-1}
             \end{pmatrix}
         + K^{2 \times 1},
    \end{equation}
where as before $A = n-2-a \in (0, 1)$, by \eqref{E:a.tween.c.and.n-2},
and $K^{2 \times 1}$ is an affine function of $\theta_{2}, \ldots, \theta_{n-2}$, and $\alpha$. 
At $x = x_{a}(\theta_{0 \ell}, \alpha_{0})$ we have $M_{a}(x) = \{ y_{0}, y_{1}, y_{n-1} \}$ and the LHS of \eqref{E:matrix.version.of.eqns} is 0. But the matrix $\begin{pmatrix}
           -A & 2 \\
           -2 & A
         \end{pmatrix}$ 
 has full rank. (The determinant is $4-A^{2} > 3$.) 
 
 Let $j \in \{ 0, 1, n-1 \}$ be arbitrary. Then, by making tiny changes in $\theta_{1}, \theta_{n-1}$, we can make $G_{a}(y_{j};x)$ the smallest of $G_{a}(y_{0};x), G_{a}(y_{1};x), G_{a}(y_{n-1};x)$ without creating any more minima. The point, $x'_{j} \in \D$, corresponding to the changed $\theta_{i}$'s will belong 
 to $\D'$ and we can make $m_{a}(x'_{j})$ arbitrarily close 
 to $y_{j}(\theta_{0 \ell}, \alpha_{0})$. Hence, for every neighborhood $\clU$  of $x_{a}(\theta_{0 \ell}, \alpha_{0})$ we will have 
 $\bigl\{ y_{0}(\theta_{0 \ell}, \alpha_{0}), y_{1}(\theta_{0 \ell}, \alpha_{0}), 
 y_{n-1}(\theta_{0 \ell}, \alpha_{0}) \bigr\} \subset \overline{ m_{a}( \clU \cap \D' ) }$. This proves the claim that for any $v = y_{0}, y_{1}$, or $y_{n-1}$,  there is 
 $\{ x_{\gamma} \} \subset \D'$ 
 s.t.\ $x_{\gamma} \to x_{a}(\theta_{\ell}, \alpha)$ 
 and $m_{a}(x_{\gamma}) \to y_{j}(\theta_{0 \ell}, \alpha_{0})$. Thus, if $\clU$ is a neighborhood of $x_{a}$ then 
 $M_{a}(x_{a}) \subset \overline{m_{a}(\clU \cap \D')}$. 
But $M_{a} \bigl( x_{a}(\theta_{\ell}, \alpha) \bigr)$ does not lie in any open semicircle. 
Therefore $x_{a}(\theta_{\ell}, \alpha) \in \Ss^{\msf{V}_{\pi/2}}$.

The set of $(\theta_{1}, \ldots, \theta_{n-1}, \alpha) \in \RR^{n}$ satisfying \eqref{E:matrix.version.of.eqns} has codimension 2. That is because arbitrary tiny perturbations to $\theta_{2}, \ldots, \theta_{n-2}$, 
and $\alpha$ can be compensated for by manipulations of $\theta_{1}$ 
and $\theta_{n-1}$ that preserve \eqref{E:matrix.version.of.eqns}, without creating any new minima. 
Thus, there exists a submanifold, $\tilde{\Ss}_{a}$, of $\D$ of codimension 2 s.t.\ 
$x_{a}(\theta_{\ell}, \alpha) \in \tilde{\Ss}_{a} \subset \Ss^{\msf{V}_{\pi/2}}$.
By theorem \ref{T:Hausdorff.measure.gives.right.value.on.manifs}, $\tilde{\Ss}_{a}$ has positive $\Hm^{n \nvar-\nvar-1}$-measure, where now $\nvar = 1$. 

Write $x_{a} = x_{a}(\theta_{0 \ell}, \alpha_{0})= (y_{1}, \ldots, y_{n})$. We have seen 
that $\zeta, \omega = O(n-2-a)$. Therefore, by    \eqref{E:thetas.zeta.omega}, $\theta_{i} = O(n-2-a)$ for $i  \in \NN_{n-1}$. And we have also seen $\theta_{n-1}, \alpha = O(n-2-a)$. 
But $y_{i} := \exp( \theta_{i} \sqrt{-1})$ ($i \in \NN_{n}-1$) and
$y_{n} := \exp\bigl[ (\pi-\alpha) \sqrt{-1} \, \bigr]$. We conclude 
    \begin{multline}  \label{E:yi.angles.O(n-2-a)}
      \angle \bigr( y_{i}, (0,1) \bigr) = O(n-2-a) \quad (i \in \NN_{n}-1) \text{ and } \\
        \angle \bigr( y_{n}, (-1,0) \bigr) = \angle(y_{n}, y_{0}) = O(n-2-a) 
          \text{ as } a \uparrow n-2.
    \end{multline}
Let $x_{0} := \bigl( (1,0), \ldots, (1,0), y_{0} \bigr)$. So $n-1$ observations in $x_{0}$ are $(1,0)$ and one is $y_{0} = (-1,0)$. Then $x_{0} \in \Pf_{k} = \Pf_{1}$. This means
\begin{equation*}
  \text{distance } (x_{a}, \Pf_{1}) \leq \rho(x, x_{0}) = O(n-2-a)
    \text{ as } a \uparrow n-2.
\end{equation*}

Let $\epsilon > 0$ and let $\mcl{B}_{\epsilon}(x_{0})$ be the open $\rho$-ball (see \eqref{E:rho.is.top.metric}) about $x_{0}$ with radius $\epsilon$. Choose $a < n-2$ so large that $x_{a} \in \mcl{B}_{\epsilon}(x_{0})$. Then $\tilde{\Ss}_{a} \cap \mcl{B}_{\epsilon}(x_{0})$ is a non-empty open subset of $\tilde{\Ss}_{a}$. 
But $\tilde{\Ss}_{a} \cap \mcl{B}_{\epsilon}(x_{0})$ still has positive $\Hm^{n \nvar-2}$-measure. We conclude that \eqref{E:dist.to.Pk.O(n-2-a)} holds.
  \end{proof}  
  
  \begin{proof}[Proof of proposition \ref{P:aug.direct.mean.beats.robst.loc}] 
By \eqref{E:ma.satisfies.hyps.of.loose.lemma}, $m_{a}$ satisfies the hypotheses of lemma \ref{L:loose.exactness.of.location.fit}. By proposition \ref{P:severe.sings.of.ma.dist} (here is where we use $n > 3$) and  \eqref{E:ma.has.no.svr.sings.near.Pk}, given $R > 0$, by making 
$a \in (0, n-2)$ sufficiently close to $n - 2$, we get $m_{a} \in F_{R, 1}$. Thus, $F_{R, 1}$ is non-empty. This proves part \ref{I:F.R.1.nonempty} of the proposition.

Let $\Phi \in F_{R, 1}$ and let $r = dist_{n-2} (\Ss^{\msf{V}_{\pi/2}}, \Pf_{1})$. By 
lemma \ref{L:solve.rho.t=r}, for $R/\delta$, and hence $r/\delta$, sufficiently small, there exists $t_{r} \in (0, n)$ s.t.\ 
    \begin{equation}  \label{E:dist(Ss.mu.t.R,T)=R/delta}
      \rho_{t_{r}} := dist_{n - 2}(\Ss_{\mu_{t_{r}}}, \T) = r/\delta.
    \end{equation}
By lemma \ref{L:solve.rho.t=r} again, we have
        \begin{equation}  \label{E:tR.approx}
          t_{r} = n - \tfrac{1}{2} ( r/\delta )^{2} + O(r^{4}/\delta^{4}),
        \end{equation} 
as asserted in part \ref{I:mu.a.beats.the.shit.out.of.Phi} of the proposition. 

By definition of $r$,
	\begin{equation*}
	  dist_{n-2} \bigl(\Ss_{\mu_{t_{r}}}, \T \bigr) = \delta^{-1} 
            dist_{n-2} (\Ss^{\msf{V}_{\pi/2}}, \Pf_{1} ) .
	\end{equation*} 
This is the first part of \eqref{E:robust.dist.vol.ineqs}.

By definition of $F_{R, 1}$, $\Phi$ satisfies the hypotheses of lemma \ref{L:loose.exactness.of.location.fit}. It follows that:
    \begin{enumerate}
        \item \eqref{E:D'.=.D.less.S} holds for $\Phi$, which means that $\Ss^{\msf{V}_{\pi/2}}$ has empty interior, 
        \item $\Phi$ is continuous on $\Pf_{1}$, which means $\Phi \restriction_{\T}$ is continuous,
        \item \eqref{E:no.severe.sings.in.Pk} holds for $\Phi$:
        $\Ss^{\msf{V}_{\pi/2}} \cap \Pf_{k} = \varnothing$. 
        \emph{A fortiori,} \eqref{E:loc.meas.has.no.severe.sings.near.T} holds:
        $\Ss^{\msf{V}_{\pi/2}} \cap \T = \varnothing$.
        \item \eqref{E:exactness.of.fit.loose} holds with $k=1$, which means \eqref{E:location.Phi.on.diagnl.loose} holds.
        \item \eqref{E:assume.Phi.S'.sym} holds.
    \end{enumerate}
 
By \eqref{E:SV.is.closed}, $\Ss^{\msf{V}_{\pi/2}}$ is closed. By remark \ref{R:symmetry}, 
$\Ss^{\msf{V}_{\pi/2}}$ is $S_{n}$-invariant. ($S_{n}$ is the symmetric group on $\{ 1, \ldots, n \}$; \eqref{E:symbol.for.symmetric.group}.) 
Therefore, by lemma \ref{L:loose.exactness.of.location.fit}, there is a measure of location, $\Omega$, on $S^{1}$ symmetric in its arguments, having exact fit order 1, and continuous off $\Ss^{\msf{V}_{\pi/2}}$. Let $\Ss_{\Omega}$ denote the singular set 
of $\Omega$. We introduce $\Omega$ because:
    \begin{enumerate}
        \item $\Ss_{\Omega} \subset \Ss^{\msf{V}_{\pi/2}}$. This means 
          \label{I:Omega.sing.set.in.S.pi/2}
        \item $\Omega$ is continuous on $\T$. \label{I:Omega.is.cont.on.T}
        \item $(\Omega, \Ss^{\msf{V}_{\pi/2}}, S_{n}, \T, n-2 = n \nvar - \nvar - 1)$ satisfies part \ref{I:Phi.is.nice} of property \ref{Pty:agree.near.T}.
          \label{I:Omega.satisfies.agree.near.T}
    \end{enumerate}

Let $\Psi : \D \setminus \tilde{\Ss} \to S^{n}$ be continuous and $S_{n}$-invariant, where $\tilde{\Ss} \subset \D$ is $S_{n}$-invariant and closed with empty interior, s.t.\ $\tilde{\Ss} \cap \T = \Ss_{\Omega} \cap \T \subset 
\Ss^{\msf{V}_{\pi/2}} \cap \T \subset \Ss^{\msf{V}_{\pi/2}} \cap \Pf_{1} = \varnothing$ and the restrictions 
$\Psi \restriction_{\T}$ and $\Omega \restriction_{\T}$ are equal. I.e., \eqref{E:assume.Phi.S'.sym}, \eqref{E:loc.meas.has.no.severe.sings.near.T}, and \eqref{E:location.Phi.on.diagnl.loose} hold for $\tilde{\Ss}$ and $\Psi$. Moreover, by property \ref{I:Omega.is.cont.on.T} of $\Omega$, we have that $\Psi$ is continuous 
on $\T$. Then, $\Psi$ satisfies the hypotheses of corollary \ref{C:Haus.measure.of.severe.sings.on.spheres} and remark \ref{R:loose.exactness.of.fit.and.homotopy} with $\nvar = 1$. Therefore, $\Hm^{n-2}(\tilde{\Ss}) > 0$. Putting this together with $\Omega$ property \ref{I:Omega.satisfies.agree.near.T} above, we see that $(\Omega, \Ss^{\msf{V}_{\pi/2}}, S_{n}, \T, n-2)$ has property \ref{Pty:agree.near.T}.

We wish to apply theorem \ref{T:lwr.bnd.on.Haus.meas} to $\Omega$ with $\Pf_{1}$ playing the role of $\Pf$. From \eqref{E:dim.Pk.=.k+1} we know that 
$p := \dim \Pf_{1} = 2$. \emph{But $\Pf_{1}$ is not a manifold.} It is the union of tori whose intersection is $\T$. (See \eqref{E:T.=.intersection.of.all.lobes} and \eqref{E:Pf1=union.of.lobes}.) Hence, we cannot rely on example \ref{Ex:TubeNbhdSpecialCaseOfCones} to tell us that $\Pf_{1}$ has a neighborhood in $T \D \restriction_{\Pf}$ fibered over $\Pf_{1}$ by cones as in definition \ref{D:fibering.by.cones}. However, from appendix \ref{Chptr:rob.loc.circle.cones.appendix}, we see that nonetheless $\Pf_{1}$ has a neighborhood in $T \D \restriction_{\Pf_{1}}$ fibered over $\Pf_{1}$ by cones. Recalling \eqref{E:prod.of.spheres.has.perm.invar.triangulation} and \eqref{E:Riem.metric.is.G.invar}, we see that theorem \ref{T:lwr.bnd.on.Haus.meas}, with $a = n-2$ and $\Pf = \Pf_{1}$, applies to $\Omega$. (``$a$'' means something different in chapters \ref{Chptr:aug.direct.mean} and \ref{Chptr:robst.loc.on.circle} than it does in chapter \ref{Chptr:Haus.meas.of.sing.set}. Sorry. It's the chapter \ref{Chptr:Haus.meas.of.sing.set} meaning that I'm using here.)
 
By the above property \ref{I:Omega.sing.set.in.S.pi/2} of $\Omega$, we have 
$\Hm^{n - 2}(\Ss_{\Omega}) \leq \Hm^{n - 2}(\Ss^{\msf{V}_{\pi/2}})$ and 
\linebreak 
$R_{\Omega} := dist_{n-2}(\Ss_{\Omega}, \Pf_{1}) 
\geq r := dist_{n-2}(\Ss^{\msf{V}_{\pi/2}}, \Pf_{1})$. 
Hence, as in \eqref{E:Hm.a.S.geq.R.d-p-1}, there exists 
$\gamma > 0$, not depending on $R$, $\Omega$ (or $\Phi$) s.t.\
    \begin{equation}  \label{EHm.n-2.S.pi/2.lwr.bnd}
      \Hm^{n-2}(\Ss^{\msf{V}_{\pi/2}}) \geq \Hm^{n - 2}(\Ss_{\Omega}) 
        \geq \gamma R_{\Omega}^{\dim \D - \dim \Pf_{1} - 1} 
          \geq \gamma r^{n-2-1} .
    \end{equation}

As $R \to 0$, we have $r \to 0$. By \eqref{E:tR.approx}, $t_{r} \to n$ as $R \to 0$. Therefore, 
by \eqref{E:R.n.nvar-nvar-1.H.bdd.above}, \eqref{E:dist(Ss.mu.t.R,T)=R/delta}, and the fact that $\nvar = 1$, for $R$ sufficiently small there exists $L > 0$ s.t.\ 
	  \begin{equation}  \label{E:R.Hm.1/L.ineq}
	    (r/\delta)^{-(n-2)} \Hm^{n -2}(\Ss_{\mu_{t_{r}}})  < 1/L 
	      \text{ so } \Hm^{n -2}(\Ss_{\mu_{t_{r}}}) < (r/\delta)^{n-2}/L .
	  \end{equation}
Taking the ratio of \eqref{EHm.n-2.S.pi/2.lwr.bnd}  to \eqref{E:R.Hm.1/L.ineq}, 
	\begin{equation} \label{E:mua.better.than.Phi}
		\frac{ \Hm^{n - 2}(\Ss^{\msf{V}_{\pi/2}}) }
		  { \Hm^{n - 2}(\Ss_{\mu_{t_{r}}}) }
			   >  \gamma \, r^{n-2-1} \times L (\delta/r)^{n-2} 
			    =  \frac{\gamma L}{r} \delta^{n-2} .
	\end{equation} 
With $r$ sufficiently small, $\gamma L/r \geq 1$.
Part \ref{I:mu.a.beats.the.shit.out.of.Phi} of the proposition follows.
  \end{proof} 

\chapter{Data and calculations for figure \ref{F:CircleSing}.} \label{Chptr:F:CircleSing.data.and.calcs}

\section{Data}
Here are the data plotted in figure \ref{F:CircleSing}. (Those data are not real. They're simulated.) ``a'' means panel ``(a)''. ``b'' means panel ``(b)''. ``x'' means $x$ coordinate. ``y'' means $y$ coordinate.
\begin{small}
\begin{verbatim}

xdata.a 
#  [1]  0.9879305 -0.5545916 -0.4624037  0.4093560 -0.6412064 -0.2475330
#  [7] -0.9788903  0.6857391 -0.1204257 -0.6383888  0.6654187  0.5204014
# [13] -0.1913034 -0.5398914 -0.6976074  0.8034177  0.9999742

ydata.a
#  [1] -0.154897545  0.832122670 -0.886669495  0.912374726  0.767368455
#  [6]  0.968879451  0.204386542 -0.727847475 -0.992722342  0.769714085
# [11]  0.746470349  0.853921796  0.981530948 -0.841734683  0.716480273
# [16]  0.595415869 -0.007184685


xdata.b 
#  [1] -0.009423075  0.273729100 -0.476790126  0.417194423  0.048214267
#  [6]  0.260129155  0.078096753 -0.246307028 -0.331049593  0.044100261
# [11]  0.663819872 -0.389348937 -0.153286581 -0.223438230  0.357478941
# [16] -0.036839568 -0.276285174

ydata.b
#  [1] 0.9999556 0.9618068 0.8790172 0.9088173 0.9988370 0.9655738 0.9969458
#  [8] 0.9691919 0.9436134 0.9990271 0.7478925 0.9210903 0.9881818 0.9747181
# [15] 0.9339212 0.9993212 0.9610757

# For convenience in working in R with these values:

xdata.a <- c(0.9879305, -0.5545916, -0.4624037, 0.4093560, -0.6412064, -0.2475330,
-0.9788903,  0.6857391, -0.1204257, -0.6383888,  0.6654187,  0.5204014,
-0.1913034, -0.5398914, -0.6976074,  0.8034177,  0.9999742)

ydata.a <- c(-0.154897545, 0.832122670, -0.886669495, 0.912374726, 0.767368455,
0.968879451, 0.204386542, -0.727847475, -0.992722342, 0.769714085,
0.746470349, 0.853921796, 0.981530948, -0.841734683, 0.716480273,
0.595415869, -0.007184685)

xdata.b <- c(-0.009423077, 0.273729103, -0.476790129, 0.417194428, 0.048214265, 
0.260129159, 0.078096750, -0.246307023, -0.331049589, 0.044100259, 0.663819868, 
-0.389348940, -0.153286579, -0.223438230, 0.357478936, -0.036839573, -0.276285169)

ydata.b <- c(0.9999556, 0.9618069, 0.8790172, 0.9088173, 0.9988370, 0.9655739,
0.9969458, 0.9691918, 0.9436134, 0.9990271, 0.7478925, 0.9210903, 0.9881817, 
0.9747181, 0.9339212, 0.9993211, 0.9610757)

\end{verbatim}
\end{small}

\section{Singular sets for figure 1.7}  \label{SS:17.point.aug.mean.exml.sings}
$\Phi_{1}$ and $\Phi_{2}$ in figure \ref{F:CircleSing} are augmented directional means: 
$\Phi_{i} = \mu_{y_{0}, a_{i}, 17}$, ($i=1,2$). (See \eqref{E:aug.direct.mean.defn}.) Here,
    \begin{equation}   \label{E:y0.a1.a2}
      n = 17, \nvar = 1, y_{0} = (0, -1), a_{1} := 4.737609, \text{ and }  a_{2} := 16.14899 .
    \end{equation}
($a_{1}$ and $a_{2}$ are exactly the values of $a$ that make the data in the figures singular.)

Therefore, by \eqref{E:bunch.of.neg.y0s.is.closest.to.T}, the Euclidean distance from the singular set $\Ss_{a_{1}}$ to $\T$ is 4.95225 and the Euclidean distance from the singular set $\Ss_{a_{2}}$ to $\T$ is 1.304615.

  \begin{remark}[Remoteness from $\Pf$] \label{R:remoteness.from.Pf}
The singular set shown in panel (b) of figure \ref{F:CircleSing} is close to the set of perfect fits for the location problem on a circle. I.e., it is close in \emph{space}. However, the augmented directional mean has a Bayesian interpretation (Nu\~{n}ez-Antonio and Gutierrz-Pena \cite{gN-AeG-P05.BayesDirectData}), and from that point of view the large size of $a_{2}$ indicates strong \emph{a priori} belief that the location of the population the data are hypothetically sampled from is near $y_{0}$. Since this data set is far from $y_{0}$, those data are far from $\Pf$ in \emph{a priori} \emph{probability}.
  \end{remark}
 
By \eqref{E:aug.mean.sing.vol.asymp}, we have 
    \begin{equation}  \label{E:dim.Sa=15}
       \dim \Ss_{a} = n \nvar - \nvar-1 = 15 .
    \end{equation} 
Computing the $(n \nvar - \nvar-1)$-dimensional Hausdorff measure of the singular set of an augmented directional mean amounts to computing a rather unpleasant integral.  
Below, specifically at \eqref{E:measure.ratio.is.6.6.billion}, we use Monte Carlo to estimate the volumes of the singular sets of $\Phi_{1}$ and $\Phi_{2}$ and show that  
the singular set of $\Phi_{1}$ has $15$-dimensional Hausdorff measure over 6 billion 
times that of $\Phi_{2}$! 

Temporarily let $n = 3, 4, \ldots$ and let $\D_{n} = (S^{1})^{n}$ be the data space consisting of data sets of $n$ observations on the unit circle. By 
\eqref{E:Sa=Sa'} and \eqref{E:S'a.D'a.defns}, a data set  
	\begin{equation} \label{E:when.aug.mean.sing}
	  (y_{1}, \ldots, y_{n}) \in \D_{n} 
	    \text{ is a singularity of } \mu_{y_{0}, a,n} \text{ if and only if }
		y_{1} + y_{2} = - \left( \sum_{i=3}^{n} y_{i} \right) -a y_{0} .
	\end{equation}
	
Let $z := y_{1} + y_{2} = - \left( \sum_{i=3}^{n} y_{i} \right) -a y_{0} \in \RR^{2}$. Then 
    \begin{equation} \label{E:|z|<2}  
        \text{If $(y_{1}, \ldots, y_{n}) \in \Ss_{a}$ then } |z| \leq 2 .
    \end{equation}
Conversely, we \emph{claim:}
	\begin{multline}  \label{E:exactly.2.pairs}
		\text{If } 0 < |z| < 2 \text{ then there exist at most two distinct pairs, } 
		   (y_{1}, y_{2}) \text{ and } (y_{2}, y_{1}) \in (S^{1})^{2} \\
		      \text{ s.t.\ } y_{1} + y_{2} = z. \text{ Moreover, } y_{1} \cdot z, > 0, \; 
		         y_{2} \cdot z > 0. 
		           \text{ If } |z| = 2 \text{ then } y_{1} = y_{2} = \tfrac{1}{2} z ,
	\end{multline} 
where as usual ``$\cdot$'' denotes the usual inner product on, in this case, $\RR^{2}$. Suppose $0 < |z| < 2$ and $y_{1} + y_{2} = z$. Write $y_{1} = (y_{11}, y_{12})$, 
$y_{2} = (y_{21}, y_{22})$, and $z = (z_{1}, z_{2})$. Observe that to prove \eqref{E:exactly.2.pairs} it suffices to consider the case $z_{2} = 0$. For suppose \eqref{E:exactly.2.pairs} holds with $z_{2} = 0$ and let $z \neq 0$ satisfying \eqref{E:|z|<2} be arbitrary. Suppose $y_{1}, y_{2}, y_{1}', y_{2}' \in S^{1}$, $y_{1} + y_{2} = z$, and $y_{1}' + y_{2}' = z$. Let $C^{2 \times 2}$ be an orthonormal matrix s.t.\ the second coordinate of $z C$ is 0. We have $y_{1} C + y_{2} C = z C$, and 
$y_{1}' C + y_{2}' C = z C$. By hypothesis, $y_{1}' C = y_{2} C$ and $y_{2}' C = y_{1} C$. Thus, $y_{1}' = y_{2}$ and $y_{2}' = y_{1}$. Also by hypothesis, 
$0 < (y_{i} C) \cdot (z C) = y_{i} C C^{T} z^{T} = y_{i} \cdot z$. 

Let $z = (z_{1},0)$ satisfy \eqref{E:|z|<2} and $0 < |z_{1}| < 2$. Suppose 
$y_{1} + y_{2} = z$. Then $y_{22} = - y_{12}$. Since $|y_{1}| = 1 = |y_{2}|$, we therefore have $y_{21} = \pm y_{11}$. Since, $|z_{1}| > 0$, we have $y_{21} = y_{11}$. 
Therefore, $y_{11} = y_{21} = z_{1}/2$. In particular, 
$y_{1} \cdot z, \, y_{2} \cdot z > 0$. We must also have $y_{12} = \pm \sqrt{1 - z_{1}^{2}/4} \neq 0$ (note that $z_{1}^{2}/4 < 1$) and $y_{22} = - y_{12} \neq 0$. Therefore, 
$y_{1} = \left( \tfrac{1}{2} z_{1}, y_{12} \right) \neq \left( \tfrac{1}{2} z_{1}, -y_{12} \right) = y_{2}$. If $|z| = 2$ then, obviously, $y_{1} = y_{2} = \tfrac{1}{2} z$. This completes the proof of the claim \eqref{E:exactly.2.pairs}. 

We can be more explicit. Suppose $z = (z_{1}, z_{2}) \in \RR^{2}$, $0 < |z| < 2$, 
$y_{1}, y_{2} \in S^{1}$, and $y_{1} + y_{2} = z$. 
Let $v = (z_{2}, -z_{1}) \in \RR^{2}$, so $v \perp z$ and $|v| = |z|$. Let $i=1,2$. Since $z \neq 0$, there exist unique $a_{i}, b_{i} \in \RR$ s.t.\ 
$y_{i} = a_{i} z + b_{i} v = (a_{i} z_{1} + b_{i} z_{2}, \, a_{i} z_{2} -  b_{i} z_{1})$. Hence, 
$z = y_{1} + y_{2} = (a_{1} + a_{2}) z + (b_{1} + b_{2}) v$ so $a_{1}+a_{2} = 1$ and 
$b_{2} = -b_{1}$. Moreover, we have 
$1 = |y_{i}|^{2} = (a_{i}^{2} + b_{i}^{2}) |z|^{2} = (a_{i}^{2} + b_{1}^{2}) |z|^{2}$. 
Hence, $|a_{1}| = \sqrt{1 - b_{1}^{2}}/|z| = |a_{2}$ and $a_{1}+a_{2} = 1$. The only solution is $a_{1} = a_{2} = 1/2$. We then have $1 = |y_{1}|^{2} = (1/4 + b_{1}^{2}) |z|^{2}$ ($i=1,2$). It follows that $-2 b_{2} = 2 b_{1} = \pm \sqrt{4/|z|^{2} - 1}$. We are free to choose $b_{1} = - (1/2) \sqrt{4/|z|^{2} - 1}$. This leads to 
    \begin{align}  \label{E:formla.for.y1.y2}
      y_{1} &= \frac{1}{2} \bigl( z - \sqrt{4/|z|^{2} - 1} \, v \bigr)
        = \frac{1}{2} \Bigl( z_{1} - \sqrt{4/|z|^{2} - 1} \, z_{2} , 
          z_{2} + \sqrt{4/|z|^{2} - 1} \, z_{1} \Bigr) , \notag \\
      y_{2} &= \frac{1}{2} \bigl( z + \sqrt{4/|z|^{2} - 1} \, v \bigr)
        = \frac{1}{2} \Bigl( z_{1} + \sqrt{4/|z|^{2} - 1} \, z_{2} , 
          z_{2} - \sqrt{4/|z|^{2} - 1} \, z_{1} \Bigr) ,  \\
               & \qquad \qquad \text{if } 0 < |z| < 2 . \notag
    \end{align}
Exchanging $y_{1}$ and $y_{2}$ leads to the other solution.

Now let $n = 17$ as in figure \ref{F:CircleSing}, suppose $3 < a < 17$, and   
	\begin{equation} \label{E:a.ain't.integer}
		a \text{  is not an integer.}
	\end{equation}
Write 
    \begin{multline}  \label{E:z(w).defn}
      z := z_{a}(w) := z_{a}[x] := - \left( \sum_{i=3}^{n} y_{i} \right) -a y_{0}, \\
          \qquad w = (y_{3}, \ldots, y_{n}) \in \D_{15} , \;
            x = (y_{1}, y_{2}, w) \in \D_{17} .
    \end{multline} 
We need to dispose of the cases $\bigl| z_{a}(w) \bigr| = 0$ or 2. We show that the collection of data sets for which $z = 0$ or $|z| = 2$ 
has $\Hm^{n \nvar - \nvar - 1 = 15}$-measure 0. For $t \in [0,2]$, let 
    \begin{equation}.  \label{E:Sa.t.defn}
      \Ss_{a}^{t} := \Bigl\{ x \in \Ss_{a} : \bigl| z_{a}[x] \bigr| = t \Bigr\}. 
    \end{equation}

\emph{First, we show } 
	\begin{equation} \label{E:Hm.15.S.2.=.0}
		\Hm^{15}(\Ss_{a}^{2}) = 0. 
	\end{equation}

For $w \in \D_{15}$ let 
    \begin{equation*}
      \zeta(w) = \bigl| z_{a}(w) \bigr|^{2} .
    \end{equation*} 
We show that $\nabla \zeta$ does not vanish on $\zeta^{-1}(4)$. 
We may parametrize $y_{i}$ 
by $( \cos \theta_{i}, \sin \theta_{i} )$, ($\theta_{i} \in \RR$, $i = 3, \ldots, 17$). 
Write $y_{0} = ( \cos \theta_{0}, \sin \theta_{0} )$. We have 
	\begin{multline*}
	  \frac{\partial}{\partial \theta_{j}} \zeta(w) 
	  = \frac{\partial}{\partial \theta_{j}} \bigl( z_{a}(w) \cdot z_{a}(w) \bigr) \\
          = \left( \frac{\partial}{\partial \theta_{j}} z_{a}(w) \right) \cdot z_{a}(w) 
           + z_{a}(w) \cdot \left( \frac{\partial}{\partial \theta_{j}} z_{a}(w) \right)
	   = 2 z_{a}(w) \cdot ( \sin \theta_{j}, - \cos \theta_{j} ) \\
		         \quad ( j = 3, \ldots, 17 ) .
	\end{multline*}
But $( \sin \theta_{j}, - \cos \theta_{j} )$ spans the subspace of $\RR^{2}$ that is orthogonal to $y_{j}$. Hence, since $|z_{a}(w)| = 2 \neq 0$ by assumption, 
we have $\tfrac{\partial}{\partial \theta_{j}} \zeta(w)  = 0$ if and only if $y_{j}$ is proportional to $z_{a}(w)$. I.e., $\nabla \zeta(w)  = 0$ if and only if 
	\[
		y_{j} = \pm | z_{a}(w) |^{-1} z_{a}(w) \qquad (j = 3, \ldots, 17). 
	\]
Let $k = 0, 1, \ldots, 15$ be the number of plus signs  (in which case let $\epsilon := -1$) or minus signs ($\epsilon := +1$) in the preceding, whichever is smaller. Then, if $\nabla \zeta(w) = 0$, 
	\begin{align}  \label{E:z.in.y3.y0.k}
		z_{a}(w) &= -\sum_{i=3}^{17} y_{i} -a y_{0} \notag \\
		  &= -\sum_{i=3}^{17} \bigl[ \pm | z_{a}(w) |^{-1} z_{a}(w) \bigr] 
		    -a y_{0} \\
		  &= - | z_{a}(w) |^{-1} \left[ \sum_{i=3}^{17} \pm 1  \right] z_{a}(w) 
		    -a y_{0} \notag \\
		  &= - \epsilon | z_{a}(w) |^{-1} \bigl[ (15 - k) - k  \bigr] z_{a}(w) 
		    -a y_{0} \notag \\
		  &= - \epsilon (15 - 2k) | z_{a}(w) |^{-1} z_{a}(w) - a y_{0}. \notag 
	\end{align}
Let 
	\begin{equation*}  \label{E:m.17.k} 
		m := - \epsilon (15 - 2k)
	\end{equation*}
as in \eqref{E:z.in.y3.y0.k}. So $m$ is an integer. Then, from \eqref{E:z.in.y3.y0.k}, 
	\begin{equation*}
		| z_{a}(w) |^{-1} \Bigl( | z_{a}(w) | - m \Bigr) z_{a}(w) = - a y_{0}.
	\end{equation*}
Therefore, 
	\begin{equation*}
		\Bigl( | z_{a}(w) | - m \Bigr)^{2} = a^{2}.
	\end{equation*}
I.e., $| z_{a}(w) | - m = \pm a$. Hence, by \eqref{E:a.ain't.integer}, $| z_{a}(w) |$ cannot be an integer. 
In particular, if $\nabla \zeta = 0$ then $| z_{a}(w) | \neq 2$. Therefore, $\nabla \zeta$ does not vanish 
on $\zeta^{-1}(4)$. Now, $$\D_{15} = (S^{1})^{15}$$ is a 15-dimensional smooth manifold. Therefore, by Boothby \cite[Theorem (5.8), p.\ 79]{wmB75}, we conclude that
	\[
		Z_{2} := \Bigl\{ w \in \D_{15} : \bigl| z_{a}(w) \bigr| 
		  = 2 \Bigr\} = \zeta^{-1}(4)
	\]
is a 14-dimensional submanifold of $\D_{15}$ and, hence, 
by \eqref{E:Haus.dim.s-manif.=.s}, 
has $\Hm^{15}$ measure 0. But, by \eqref{E:exactly.2.pairs}, 
	\begin{equation*} 
		\Ss_{a}^{2} = \Bigl\{ (\tfrac{1}{2} z_{a}(w), 
		  \tfrac{1}{2} z_{a}(w), w) \in \D_{17} : w \in Z_{2} \Bigr\}
	\end{equation*}
and $w \mapsto (\tfrac{1}{2} z_{a}(w), \tfrac{1}{2} z_{a}(w), w)$ is Lipschitz on $Z_{2}$. Therefore, by \eqref{E:Lip.magnification.of.Hm}, \eqref{E:Hm.15.S.2.=.0} holds, as desired.

\emph{Next, we show that} 
    \begin{equation} \label{E:Hm.15.S.0.=.0}
      \dim \Ss_{a}^{0} < n \nvar - \nvar -1 = 15, \text{ so }
        \Hm^{15}(\Ss_{a}^{0}) = 0 .
    \end{equation} 
Suppose $x = (y_{1}, \ldots, y_{17}) \in \Ss_{a}^{0}$. \emph{Claim:}
	\begin{equation} \label{E:y3.17.not.identical}
	  \text{ There exists } j = 4, \ldots, 17 \text{ s.t. } y_{j} \text{ and } y_{3}
	      \text{ are not linearly dependent.}
	\end{equation}
The argument is similar to that which led to \eqref{E:z.in.y3.y0.k}. Suppose \eqref{E:y3.17.not.identical} is false. Then for every $j = 4, \ldots, 17$ there exists 
$c_{i} \neq 0$ s.t.\ $y_{j} = c_{j} y_{3}$. But $|y_{j}| = 1 = |y_{j}|$. 
Therefore, $c_{i} = \pm 1$. But $x = (y_{1}, \ldots, y_{17}) \in \Ss_{a}^{0}$ so 
$\sum_{i=3}^{17} y_{i} + a y_{0} = - z = 0$ and so for some $m \in \ZZ$, 
$m y_{3} = a y_{0}$. Thus, $|m| = a$. But this is impossible by \eqref{E:a.ain't.integer}. \eqref{E:y3.17.not.identical} is proved. 

We take an approach similar to that used in section \ref{SS:aug.mean.sing.set.size}. Write $y_{i} = (y_{i1}, y_{i2})$ ($i=0, \ldots, 17$). Let $U \subset (\RR^{2})^{15}$ consist of all 15-tuples $w = (y_{3}, \ldots, y_{17}) \in (\RR^{2})^{15}$ s.t.\ \eqref{E:y3.17.not.identical} holds and in addition, no $y_{j} = 0$ ($j = 3, \ldots, 17$). Then $U$ is open in $\RR^{30}$.  
Consider the map $H : U \to \RR^{17}$ given by
	\begin{align}
	      H(w) &= H(y_{3}, \ldots, y_{17}) \notag \\      
	         :&= \bigl( z_{a}(w), \, |y_{3}|^{2}, \ldots, |y_{17}|^{2} \bigr)   \notag \\
	         &= \left( - a y_{01} - \sum_{i=3}^{17} y_{i1}, - a y_{02} 
	              - \sum_{i=3}^{17} y_{i2}, \, |y_{3}|^{2},
	                 \ldots, |y_{17}|^{2} \right) \in \RR^{17} , \\ 
	         &\qquad \qquad (y_{3}, \ldots, y_{n}) \in U'.    \notag
	\end{align}
$H$ is smooth on $U$. Let $u := (0, 0, 1, \ldots, 1) \in \RR^{17}$ and 
$N := H^{-1}(u) \subset U$. 

The map $(y_{1}, w) \mapsto (y_{1}, -y_{1}, w)$ ($y_{1} \in S^{1}, \, w \in N$) is a Lipschitz homeomorphism of $S^{1} \times N$ onto $\Ss_{a}^{0}$. Thus, by \eqref{E:Lip.magnification.of.Hm} again, it suffices to show that 
$\Hm^{15} (S^{1} \times N) = 0$. \emph{Suppose} we show that $H$ has full rank 17 on $U$. It would then follow, by Boothby \cite[Theorem (5.8), p.\ 79]{wmB75} that  $N$ is a smooth manifold of dimension 30 - 17 = 13. Hence, by Boothby \cite[Theorem (1.7), p.\ 57]{wmB75},  
$S^{1} \times N$ is a $(1+13)$-dimensional differentiable manifold. It follows that 
$\dim \Ss_{a}^{0} < 15$ as desired. In summary, we conclude that we may assume
    \begin{equation}   \label{E:0<|z|<2}
      0 < |z| < 2 . 
    \end{equation}

Now we confirm that $H$ has full rank 17 on $N'$. Regarding each $y_{i}$ as a $1 \times 2$ row matrix, the Jacobian matrix (Boothby \cite[p.\ 26]{wmB75}) of $H$ is given by
        \begin{equation*}
	      DH(y_{3}, \ldots, y_{n})^{17 \times 30} =
	         \begin{pmatrix}
	            -I_{2} & -I_{2} & \cdots & -I_{2} \\
	            2 y_{3} & 0^{1 \times (2)} & \cdots & 0^{1 \times (2)} \\
	            0^{1 \times (2)} & 2 y_{4} & \cdots & 0^{1 \times (2)} \\
	            \vdots &  \vdots & \ddots & \vdots \\
	            0^{1 \times (2)} & 0^{1 \times (2)} & \cdots & 2 y_{17}
	         \end{pmatrix}, \qquad
	            (y_{3}, \ldots, y_{n}) \in U,
	\end{equation*}
where $I_{2}$ is the $2 \times 2$ identity matrix. Let $w = (y_{3}, \ldots, y_{17}) \in U$ so 
no $y_{j} = 0$ and no $y_{4}, \ldots, y_{17}$ is proportional to $y_{3}$. 

Let $v^{1 \times 17} = (v_{1}, \ldots, v_{17}) \in \RR^{17}$. Then 
	\begin{multline*}
		v DH(w) = v DH(y_{3}, \ldots, y_{n}) \\
		  = (-v_{1}, -v_{2}, -v_{1}, -v_{2}, \ldots, -v_{1}, -v_{2})^{1 \times 30}
		    + 2(v_{3} y_{3}, v_{4} y_{4}, \ldots, v_{17} y_{17})^{1 \times 30}.
	\end{multline*}
Call the vector $-(v_{1}, v_{2}) + 2 v_{j} y_{j}$ the ``$j^{th}$-component of $v DH(w)$'' 
($j = 3, \ldots, 17$). Suppose $v \neq 0$. We show that $v DH(w) \neq 0$. First, suppose $(v_{1}, v_{2}) = 0$. Then for some $j = 3, \ldots, 17$ we have $v_{j}$ is not 0. 
Then the $j^{th}$-component of $v DH(w)$ is 
$2 v_{j} y_{j} \neq 0$. Next, suppose $(v_{1}, v_{2}) \neq 0$ but for some $j = 3, \ldots, 17$ we have $v_{j} = 0$. In that case the $j^{th}$-component of $v DH(w)$ is $-(v_{1}, v_{2}) \neq 0$. Finally, suppose $(v_{1}, v_{2}) \neq 0$ and for every $j = 3, \ldots, 17$ we have $v_{j} \neq 0$. Then the components of $v DH(w)$ cannot all be 0 because otherwise we would have  $2 v_{j} y_{j} = (v_{1}, v_{2}) = 2 v_{3} y_{3}$ ($j = 4, \ldots, 17$), so $y_{j} = (v_{3}/v_{j}) y_{3}$, i.e., $y_{j}$ and $y_{3}$ are linearly dependent. This contradicts $w \in U$. We conclude that $DH(w)$ has full rank 17 for $w \in U$. From the argument given above it therefore follows that $\dim \Ss_{a}^{0} < 15$ 
so $\Hm^{15}(\Ss_{a}^{0}) = 0$.

Hence, by \eqref{E:when.aug.mean.sing} and \eqref{E:exactly.2.pairs} except for a set 
of $\Hm^{17-2}$-measure 0, there is a two-to-one correspondence between the data sets 
in $\Ss_{a}$ and data sets $w = (y_{3}, y_{4}, \ldots, y_{17}) \in \D_{n-2}$ s.t.\ 
$|z| = \left| \sum_{i=3}^{n} y_{i} + a y_{0} \right| \in (0,2)$. 
 
We construct a map $F : \D_{15} \to \D_{17}$ that takes 
$w \in \D_{15}$ to one of the two points in $\Ss_{a}$ whose last 15 components constitute $w$. We begin by defining a function
$f = (f_{1}, f_{2}) : B_{2}^{2}(0) \setminus \{0\} \to (S^{1})^{2} \subset \RR^{4}$ (see \eqref{E:Euc.ball.defn}) s.t.\ 
if $z \in B_{2}^{2}(0) \setminus \{0\} \subset \RR^{2}$ then $f_{1}(z) + f_{2}(z) = z$. 
($f_{i}(z) \in S^{1}$, $i =1,2$.) 
Write $z = (z_{1}, z_{2}) \in B_{2}^{2}(0) \setminus \{0\}$ and let
    \begin{multline} \label{E:figure.F:CircleSing.f.defn}
         f(z) := 
           \frac{1}{2} \left(  z_{1} - \sqrt{(4/|z|^{2})-1} \; z_{2} , \;
              z_{2} + \sqrt{(4/|z|^{2})-1} \; z_{1}, \right. \\
               z_{1} + \sqrt{(4/|z|^{2})-1} \; z_{2} , \;
                 \left. z_{2} - \sqrt{(4/|z|^{2})-1} \; z_{1} \right) \in \RR^{4} ,
           \qquad 0 < |z| < 2 .
    \end{multline}
It is immediate that the components of $f$ are finite and real.  
Writing $f = (f_{1}, f_{2})$, with $f_{i} = (f_{i1}, f_{i2})$ having codomain $\RR^{2}$, 
\eqref{E:formla.for.y1.y2} implies 
    \begin{equation}  \label{E:f1,f2.behave.well}
      \bigl| f_{1}(z) \bigr| = \bigl| f_{2}(z) \bigr| = 1 \text{ and }
        f_{1}(z) + f_{2}(z) = z .
    \end{equation} 
Moreover, with $0 < |z| < 2$, $f_{1}(z) \neq f_{2}(z)$. Specifically, let $z = (z_{1}, z_{2}) \in B_{2}^{2}(0) \setminus \{0\}$. Then at least one of the following is true: 
$f_{21}(z) > f_{11}(z)$ or $f_{22}(z) < f_{12}(z)$. Thus, exchanging the first two coordinates with the last two, we get the other possible $(y_{1}, y_{2})$ that makes 
$(y_{1}, y_{2}, w) \in \D_{17}$ a singularity. (See \eqref{E:exactly.2.pairs}. $w \in \D_{15}$ is related to $z$ as in \eqref{E:z(w).defn}.) 

Let $g(r) := \sqrt{(4/r^{2})-1}$. Thus, 
    \begin{equation*}
      f(z) =  \frac{1}{2} \Bigl( z_{1} - g \bigl(|z| \bigr) z_{2}, \; 
        z_{2} + g \bigl(|z| \bigr) z_{1} , \; 
          z_{1} + g \bigl(|z| \bigr) z_{2}, \; z_{2} - g \bigl(|z| \bigr) z_{1} \Bigr) .
    \end{equation*}  
Let $h(r) := -4/\bigl( r^{3} \sqrt{4-r^{2}} \bigr)$ ($r \in [0, 2)$), 
so $r h(r) = g'(r)$ and 
$(\partial/\partial z_{i}) g \bigl( |(z_{1}, z_{2}| \bigr) = h \bigl( |(z_{1}, z_{2}| \bigr) z_{i}$. Writing $r := |z|$, it is easy to see that 
    \begin{equation*}
        Df(z_{1}, z_{2}) = \frac{1}{2}
            \begin{pmatrix}
               1 - h(r) z_{1} z_{2} & - h(r) z_{2}^{2} - g(r) \\
               h(r) z_{1}^{2} + g(r) & 1 + h(r) z_{1} z_{2}  \\
               1 + h(r) z_{1} z_{2} & h(r) z_{2}^{2} + g(r) \\
               -h(r) z_{1}^{2} - g(r) & 1 - h(r) z_{1} z_{2}  
            \end{pmatrix}^{4 \times 2} .
    \end{equation*}
So, the $j^{th}$ column consists of derivatives w.r.t.\ $z_{j}$. The rows consist of derivatives of $f_{11}$, $f_{12}$, $f_{21}$, $f_{22}$, resp. 

Let $w = (y_{3}, \ldots, y_{n}) \in \RR^{30}$, where 
$y_{i} = (y_{i1}, y_{i2}) \in \RR^{2}$. Thus, 
$w = (y_{3,1}, y_{3,2}, y_{4,1}, y_{4,2}, \ldots,$ \linebreak 
$y_{17,1}, y_{17,2}) \in \RR^{30}$. (Thus, we temporarily drop the requirement that 
$w \in \D_{15} = (S^{1})^{15}$.) By \eqref{E:z(w).defn}, we have 
    \begin{equation}  \label{E:za(w).defn}
      z_{a}(w) = -(y_{3,1} + y_{4,1} + \cdots + y_{17,1}, \;
        y_{3,2} + y_{4,2} + \cdots + y_{17,2}) - a y_{0} \in \RR^{2} .
    \end{equation} 
Hence,
    \begin{equation*}
       Dz_{a}(w) = 
        \begin{pmatrix}
          -1 & 0 & \ldots & -1 & 0 \\
          0 & -1 & \ldots & 0 & -1
        \end{pmatrix}^{2 \times 30} .
    \end{equation*}

Recall \eqref{E:Euc.ball.defn}. Define
    \begin{equation*}
       B'_{2}(0) := B_{2}^{2}(0) \setminus \{0\} .
    \end{equation*}
Recall $\D_{15} := (S^{1})^{15}$.  
Notice that $z_{a}^{-1} \bigl[ B'_{2}(0) \bigr] \subset \RR^{30}$ is open. Define 
    \begin{equation}  \label{E:F.maps.R30.toR34}
      F(w) = \Bigl( f \big[ z_{a}(w) \bigr], w \Bigr) \in \RR^{34} ,
        \qquad w \in z_{a}^{-1} \bigl[ B'_{2}(0) \bigr] .
    \end{equation} 
$F$ is obviously injective.

We have
    \begin{equation*}
      DF(w) =
        \begin{pmatrix}
           Df \big[ z_{a}(w) \bigr] Dz_{a}(w) \\
           I_{30}
        \end{pmatrix}^{34 \times 30} ,
          \qquad w \in z_{a}^{-1} \bigl[ B'_{2}(0) \bigr] .
    \end{equation*}
Here $I_{30}$ is the $30 \times 30$ identity matrix.

Parametrize $\D_{15}$ by 
    \begin{multline}  \label{E:psi.15}
      \psi(\theta_{3}, \ldots, \theta_{17}) := (\cos \theta_{3}, \sin \theta_{3}, 
        \ldots, \cos \theta_{i}, \sin \theta_{i}, \ldots, \cos \theta_{17}, 
          \sin \theta_{17})^{1 \times 30} , \\
            \qquad \theta_{3}, \ldots, \theta_{17} \in (-\pi, \pi] .
    \end{multline}
Since the set 
$\bigl\{ (\theta_{3}, \ldots, \theta_{17}) \in (-\pi, \pi]^{15} : 
  \theta_{i} = \pi  \text{ for some } i\bigr\}$ 
is a $\Hm^{15}$ null set, we may restrict $\psi$ to $(-\pi, \pi)^{15}$. We have
    \begin{multline*}
      D\psi(\theta_{3}, \ldots, \theta_{17}) = 
        \begin{pmatrix}
          -\sin \theta_{3} & 0 & 0 & \ldots& 0 & 0 \\
          \cos \theta_{3} & 0 & 0 & \ldots& 0 & 0 \\
          0 & -\sin \theta_{4} & 0 & \ldots& 0 & 0 \\
          0 &  \cos \theta_{4} & 0 & \ldots& 0 & 0 \\
          \vdots & \vdots & \vdots & \ldots & \vdots & \vdots  \\
          0 & 0 & 0 & \ldots & 0 &  -\sin \theta_{17}  \\
          0 & 0 & 0 & \ldots & 0 &  \cos \theta_{17}
        \end{pmatrix}^{30 \times 15} , \\
            \qquad \theta_{3}, \ldots, \theta_{17} \in (-\pi, \pi) .
    \end{multline*}

Let 
    \begin{equation}  \label{E:za.psi.and.F.psi}
      z_{a}^{\psi} := z_{a} \circ \psi \text{ and } 
        F^{\psi} := F \circ \psi .
    \end{equation}
$F$ defined in \eqref{E:F.maps.R30.toR34} is injective. Therefore, $F^{\psi}$ is injective. We get the Jacobian matrices 
    \begin{gather*}
          Dz^{\psi}(\theta) =  
            \begin{pmatrix}
                  \sin \theta_{3} & \dots & \sin \theta_{17} \\
                  -\cos \theta_{3} & \dots & -\cos \theta_{17}
            \end{pmatrix}^{2 \times 15} \\
                                  \\  
                          \text{and} \\
                                  \\  
          DF^{\psi}(\theta) =  
            \begin{pmatrix}
                  & &  Df \big[ z^{\psi}(\theta) \bigr] \; Dz^{\psi}(\theta) &  & &   \\
          -\sin \theta_{3} & 0 & 0 & \ldots& 0 & 0 \\
          \cos \theta_{3} & 0 & 0 & \ldots& 0 & 0 \\
          0 & -\sin \theta_{4} & 0 & \ldots& 0 & 0 \\
          0 &  \cos \theta_{4} & 0 & \ldots& 0 & 0 \\
          \vdots & \vdots & \vdots & \ldots & \vdots & \vdots  \\
          0 & 0 & 0 & \ldots & 0 &  -\sin \theta_{17}  \\
          0 & 0 & 0 & \ldots & 0 &  \cos \theta_{17}
            \end{pmatrix}^{34 \times 15} , \\ \\
              \qquad \theta = (\theta_{3}, \ldots, \theta_{17}) \in (-\pi, \pi)^{15} .
    \end{gather*}  
So the first four rows of $DF^{\psi}(\theta)$ are given by 
$Df \big[ z^{\psi}(\theta) \bigr] \; Dz^{\psi}(\theta)$.

By Boothby \cite[Theorem (1.6), p.\ 109]{wmB75}, $DF^{\psi}$ is the matrix of $F_{\ast} : T_{\theta} \to \RR^{34}$. What is the matrix, 
$DF^{\psi\ast}$ of the adjoint $\RR^{34} \to T_{\theta}$? 
Let $e_{1}, \ldots, e_{15}$ and $v_{1}, \ldots, v_{34}$ be the usual bases on $\RR^{15}$ and $\RR^{34}$, resp., written as row vectors. 
Let ``$\cdot$'' be the usual dot product on Euclidean space. (The Riemannian metric on $\D^{15}$ is just the restriction of the dot product. See \eqref{E:directional.location.Riem.metric}.) Suppose $c$ and $d$ are 34-dimensional column vectors, Then $c \cdot d = c^{T} d$. Let $i = 1, \ldots, 15$ and $j = 1, \ldots, 34$. Let $DF^{\psi} e_{i}^{T}$ play the role of $c$ and $v_{j}^{T}$ play the role of $d$. Then
    \begin{multline*}
      (DF^{\psi} e_{i}^{T}) \cdot v_{j}^{T} 
        = [ DF^{\psi} e_{i}^{T} ]^{T} v_{j}^{T}  
        = \bigl[ e_{i} (DF^{\psi})^{T} \bigr] v_{j}^{T} 
          = e_{i} \bigl[ (DF^{\psi})^{T} v_{j}^{T} \bigr] \\
           = \bigl[ (DF^{\psi})^{T} v_{j}^{T} \bigr]^{T} \, e_{i}^{T}
            = \bigl[ (DF^{\psi})^{T} v_{j}^{T} \bigr] \cdot e_{i}^{T} ,
    \end{multline*}
the dot product of the column vectors $(DF^{\psi})^{T} v_{j}^{T}$ and $e_{i}^{T}$. So $DF^{\psi\ast} = (DF^{\psi})^{T}$, the transpose of $DF^{\psi}$.

Define 
    \begin{equation}   \label{E:Wa.defn}
      \mcl{W}_{a} := z_{a}^{-1} \bigl[ B'_{2}(0) \bigr] \cap \D_{15} .
    \end{equation}
Let 
    \begin{equation}  \label{E:F12.defn}
      \mcl{F}_{12} := F(\mcl{W}_{a}) \subset \D_{17} .
    \end{equation}
\emph{Claim:} 
$\mcl{F}_{12} \subset \Ss_{a}$. Let $w \in \D_{15}$ and suppose 
$z_{a}(w) \in B'_{2}(0)$. I.e., $w \in \mcl{W}_{a}$. 
Write $w = (y_{3}, \ldots, y_{15})$, 
where $y_{j} = (y_{j,1}, y_{j,2}) \in S^{1}$ ($j = 3, \ldots, 17$). 
Then, by \eqref{E:f1,f2.behave.well} and \eqref{E:za(w).defn},
$f_{i} \big[ z_{a}(w) \bigr] \in S^{1}$ ($i=1,2$) and 
    \begin{multline*}
      f_{1} \big[ z_{a}(w) \bigr] + f_{2} \big[ z_{a}(w) \bigr] = 
        z_{a}(w) \\
          = -(y_{3,1} + y_{4,1} + \cdots + y_{17,1}, 
            y_{3,2} + y_{4,2} + \cdots + y_{17,2}) - a y_{0} \in \RR^{2}. 
    \end{multline*}
Hence, by \eqref{E:F.maps.R30.toR34} and \eqref{E:when.aug.mean.sing}, we have 
$F(w) \in \Ss_{a}$. The claim $\mcl{F}_{12} \subset \Ss_{a}$ is proved. 

Define $swap : \RR^{4} \to \RR^{4}$ by 
    \begin{equation*}
      swap(x_{1}, x_{2}, x_{3}, x_{4}) := (x_{3}, x_{4}, x_{1}, x_{2}) , 
        \text{ where } x_{1}, x_{2}, x_{3}, x_{4} \in \RR . 
    \end{equation*}
By the observation made after \eqref{E:figure.F:CircleSing.f.defn}, if $z, z' \in B'_{2}(0)$ (i.e.\  $|z|, |z'| \in (0,2)$), 
then $swap \circ f(z) \neq f(z')$. I.e., 
    \begin{equation}  \label{E:swap.is.disjoint}
      swap \circ f \bigl[ B'_{2}(0) \bigr] 
        \text{ and } f \bigl[ B'_{2}(0) \bigr] \text{ are disjoint. }
    \end{equation}
Let 
    \begin{multline*}
      \mcl{F}_{21} := \Bigl\{ \bigl( swap(v), w \bigr) \in \RR^{34} : 
        (v, w) = F(w), w \in \mcl{W}_{a} \Bigr\} \\
          = \Bigl\{ \bigl( swap(v), w \bigr) \in \RR^{34} : (v, w) \in \mcl{F}_{12} \Bigr\} . 
    \end{multline*}
By \eqref{E:when.aug.mean.sing}, $\mcl{F}_{21} \subset \Ss_{a}$, 
since $\mcl{F}_{12} \subset \Ss_{a}$. By \eqref{E:swap.is.disjoint},   
    \begin{equation*}
      \mcl{F}_{12} \text{ and } \mcl{F}_{21} \text{ are disjoint. }
    \end{equation*}

Recall \eqref{E:Sa.t.defn}. \emph{Claim:} 
    \begin{equation}  \label{E:F12.cup.F21.only.misses.null.set}
      \Ss_{a} \setminus (\mcl{F}_{12} \cup \mcl{F}_{21}) 
        = \Ss_{a}^{0} \cup \Ss_{a}^{2}. 
    \end{equation}
Let $x = (y_{1}, y_{2}, w) = (y_{1}, \ldots, y_{n}) 
\in \Ss_{a} \setminus (\Ss_{a}^{0} \cup \Ss_{a}^{2})$. Since $x \in \Ss_{a}$, \eqref{E:when.aug.mean.sing} holds so,  
by \eqref{E:z(w).defn}, $z := y_{1} + y_{2} = z_{a}(w)$. Therefore, by \eqref{E:formla.for.y1.y2} and \eqref{E:figure.F:CircleSing.f.defn}, we have  
$(y_{1}, y_{2}) = f(z)$ or $(y_{2}, y_{1}) = f(z)$. Since 
$x \in \Ss_{a} \setminus (\Ss_{a}^{0} \cup \Ss_{a}^{2})$, we have
$|y_{1} + y_{2}| \in (0,2)$. Hence, if $(y_{1}, y_{2}) = f(z)$ then $x \in \mcl{F}_{12}$. If $(y_{2}, y_{1}) = f(z)$ then $x \in \mcl{F}_{21}$. Therefore, if 
$x \in \Ss_{a} \setminus (\mcl{F}_{12} \cup \mcl{F}_{21})$ then we must have $x \in \Ss_{a}^{0} \cup \Ss_{a}^{2}$. I.e., 
$\Ss_{a} \setminus (\mcl{F}_{12} \cup \mcl{F}_{21}) \subset \Ss_{a}^{0} \cup \Ss_{a}^{2}$.

Conversely, let $x \in \Ss_{a}^{0} \cup \Ss_{a}^{2}$ and suppose 
$x \in \mcl{F}_{12}$. Then there exists $w \in \mcl{W}_{a}$ s.t.\ 
$f \bigl[ z_{a}(w) \bigr] = (y_{1}, y_{2})$. But by definition, \eqref{E:Wa.defn}, 
of $\mcl{W}_{a}$, 
$z_{a}(w) \in B'_{2}(0)$. I.e., $y_{1} + y_{2} \in (0,2)$. This contradicts 
$x \in \Ss_{a}^{0} \cup \Ss_{a}^{2}$. Similarly, if $x \in \mcl{F}_{21}$. This completes the proof of the claim \eqref{E:F12.cup.F21.only.misses.null.set}. 
 
Recall \eqref{E:dim.Sa=15}. Obviously, for $v \in \D_{2}$ and $w \in \D_{15}$ and 
$(v,w) \in \Ss_{a}$, we have $(v, w \bigr) \mapsto \bigl( swap(v), w \bigr)$ is an isometry and so preserves $\Hm^{15}$-measure. By \eqref{E:when.aug.mean.sing}, \eqref{E:formla.for.y1.y2}, \eqref{E:Sa.t.defn}, \eqref{E:Hm.15.S.2.=.0},  \eqref{E:Hm.15.S.0.=.0}, and \eqref{E:F12.cup.F21.only.misses.null.set}, 
    \begin{multline}  \label{E:F+,F-.are.disjoint}
      \mcl{F}_{12} \text{ and } \mcl{F}_{21} \text{ are disjoint subsets of } \Ss_{a} , \;
        \Hm^{15}(\mcl{F}_{21}) = \Hm^{15}(\mcl{F}_{12}), \text{ and } \\
          \Hm^{15} \bigl[ \Ss_{a} \setminus (\mcl{F}_{12} \cup \mcl{F}_{21}) \bigr] = 0 .
    \end{multline}

Recall \eqref{E:za.psi.and.F.psi}. As in \eqref{E:integral.over.S3}  
we can compute $\Hm^{15}(\Ss_{a})$ using the ``area formula'' (Hardt and Simon \cite[1.8 p.\ 13]{rHlS86.GMT}, Federer \cite[3.2.3, p.\ 243 and 3.2.46, p.\ 282]{hF69}): Let 
    \begin{equation*}
      J^{\psi}(w) = \sqrt{\det \bigl( (DF^{\psi\ast}) (DF^{\psi}) \bigr) } .
    \end{equation*}

Let
    \begin{equation}  \label{E:Theta.a.defn}
      \Theta_{a} := \psi^{-1}(\mcl{W}_{a}) .
    \end{equation}
Since $\psi$ is a bijection, we have $\mcl{F}_{12} = F^{\psi}(\Theta_{a})$. 
By \eqref{E:R.n.nvar-p-1.H.bdd.away.frm.0}, $\Hm^{15}(\Ss_{a}) > 0$. By \eqref{E:F+,F-.are.disjoint}, \eqref{E:F12.defn}, and \eqref{E:indicator.fn.defn}, 
	\begin{align}  \label{E:integral.over.Wa} 
		\Hm^{15}(\Ss_{a}) &= 2 \Hm^{15}(\mcl{F}_{12}) \notag \\ 
		  &= 2 \Hm^{15} \bigl[ F^{\psi}(\Theta_{a}) \bigr] \\
		  &= 2 \int_{\Theta_{a}} J^{\psi}(\theta) \; \Hm^{15} (d\theta) \notag \\
		  &= 2 \int_{(-\pi, \pi]^{15}} 1^{\Theta_{a}}(\theta) J^{\psi}(\theta) 
		       \; \Hm^{15} (d\theta) . \notag
	\end{align}

By theorem \ref{T:Hausdorff.measure.gives.right.value.on.manifs}, $\Hm^{15}$ on $\RR^{15}$ is just Lebesgue measure $\mcl{L}^{15}$. Hence,  
    \begin{equation} \label{E:Hm.15.of.D.15}
      \Hm^{15} \bigr( (-\pi, \pi]^{15} \bigl) = \mcl{L}^{15} \bigr( (-\pi, \pi]^{15} \big)
        = (2 \pi)^{15} \approx 939 \text{ billion} . 
    \end{equation}
or nearly 1 trillion. Let $V := \mcl{L}^{15} \bigr( (-\pi, \pi]^{15} \big)$. Thus, 
    \begin{equation*}
      2 V \approx 1.88 \times 10^{12} , 
    \end{equation*}
about 1.88 trillion. Let $P := V^{-1} \mcl{L}^{15} \, \measRestrict \, (-\pi, \pi]^{15}$, where 
$\mcl{L}^{15} \, \measRestrict \, (-\pi, \pi]^{15}$ is the restriction of $\mcl{L}^{15}$ 
to $(-\pi, \pi]^{15}$. $P$ is a probability measure. Let $X : (-\pi, \pi]^{15} \to \RR$ be the function 
    \begin{equation}  \label{E:Xa(theta).defn}
      X(\theta) := X_{a}(\theta) := 1^{\Theta_{a}}(\theta) \, J^{\psi}(\theta) .
    \end{equation} 
By \eqref{E:integral.over.Wa}, we then have
	\begin{equation}  \label{E:volume.is.expectation} 
		\Hm^{15}(\Ss_{a}) = 2 V \int_{\RR^{15}} X(\theta) \; P(d\theta)
		    = 2 V EX_{a} , 
	\end{equation}
where $EX_{a}$ is the expectation (population mean) of $X_{a}$ w.r.t.\ $P$. I.e., we regard $X_{a} = X_{a}(\theta)$ as a random variable with $\theta$ a random element 
of $(-\pi, \pi]^{15}$ with distribution $P$ on $(-\pi, \pi]^{15}$. I.e., $\theta$ is uniformly distributed 
on $(-\pi, \pi]^{15}$. 

We used this idea to compute the volume, $\Hm^{15}(\Ss_{a_{1}})$, by Monte Carlo. We do the computations using the R language (R Core Team \cite{RCT22.RManual}). Define $a_{1}$ as in \eqref{E:y0.a1.a2}. We simulated the random variable 
$X_{a_{1}}$ fifty million times and estimated $EX_{a_{1}}$ by the sample mean, 
$\overline{X}_{a_{1}}$ of the fifty million realized values. Denote the estimate 
by $\overline{X}_{a_{1}}$. We have
    \begin{equation}
      \overline{X}_{a_{1}} \approx 0.833 .  
    \end{equation}

Hence, we estimate 
    \begin{equation} \label{E:SS.a1.approx.vol}
      \Hm^{15}(\Ss_{a_{1}}) 
        \approx 2V \overline{X}_{a_{1}}
          \approx 1.56 \times 10^{12} , 
    \end{equation}
nearly 1.6 trillion. 
Of course, if $Y$ is a random data set uniformly distributed over $\D_{17}$, the probability that $Y \in \Ss_{a_{1}}$ is 0, because, by \eqref{E:aug.mean.sing.vol.asymp}, 
$\dim \Ss_{a_{1}} = 15 < 17 = \dim \D = \dim \D_{17}$.


Define $a_{2}$ as in \eqref{E:y0.a1.a2}. We simulated the random variable $X_{a_{2}}$ 1.5 billion times. This was not overkill because \emph{in none of the 1.5 billion simulations was $1^{\Theta_{a}}(\theta) \neq 0$!} This does not mean 
$\Ss_{a_{2}}$ is empty. After all, panel (b) of figure \ref{F:CircleSing} exhibits a data set in $\Ss_{a_{2}}$.
%
%
%

So an accelerated method for estimating $E X_{a_{2}}$ is needed. By \eqref{E:Xa(theta).defn} and \eqref{E:Theta.a.defn}, in order that $X(\theta) \neq 0$ we need $\psi(\theta) \in \mcl{W}_{a}$. Write $\psi(\theta) = (y_{3}, \ldots, y_{17}) \in \D_{15}$, $z = z_{a}\bigl[ \psi(\theta) \bigr]$, and let $v := -\sum_{i=3}^{17} y_{i} \in \RR^{2}$. 
By \eqref{E:|z|<2}, \eqref{E:z(w).defn}, and \eqref{E:Wa.defn} $a_{2} - |v| \leq |v - a_{2} y_{0}| = |z| < 2$. I.e., we need $|v| > a_{2} - 2 > 0$. Now, 
    \begin{align}
      4 > |z|^{2} &= |v - a_{2} y_{0}|^{2} \notag \\
        &= |v|^{2} - 2 a_{2} v \cdot y_{0} + a_{2}^{2} \\
        &> (a_{2}^{2} - 4 a_{2} + 4) - 2 a_{2} v \cdot y_{0} + a_{2}^{2}  \notag \\
        &= 2 a_{2}^{2} - 4 a_{2} + 4 - 2 a_{2} v \cdot y_{0} \notag .
    \end{align}
Thus, 
    \begin{equation*}
      - \sum_{i=3}^{17} y_{i} \cdot y_{0} = v \cdot y_{0} > a_{2} - 2 .
    \end{equation*} 
Since for every $i$, $y_{i} \cdot y_{0} \geq -1$, we then have
    \begin{equation*}
      -y_{3} \cdot y_{0} + 14 
        \geq - y_{3} \cdot y_{0} - \sum_{i=4}^{17} y_{i} \cdot y_{0} 
          = - \sum_{i=3}^{17} y_{i} \cdot y_{0} > a_{2} - 2 .
    \end{equation*}
I.e., $y_{3} \cdot y_{0} < 16 - a_{2}$. By \eqref{E:y0.a1.a2}, $16 - a_{2} = - 0.14899$. In general, a necessary condition that $X \neq 0$ is:
    \begin{equation*}
      \cos \angle(y_{i} , y_{0}) = y_{i} \cdot y_{0} < 16 - a_{2}, 
        \quad i = 3, \ldots, 17 .
    \end{equation*}
(See \eqref{E:angle.between.vectors}.)
 
By \eqref{E:y0.a1.a2}, $y_{0} = (0, -1)$. 
Write $y_{i} = (\cos \theta_{i}, \sin \theta_{i})$ ($i=3, \ldots, 17$). Thus, in order that $X \neq 0$ it is at a minimum necessary that $\sin \theta_{i} > a_{2} - 16$. Let $\beta := \arcsin (a_{2} - 16) \approx 0.150$. 
Then we require $\theta \in (\beta, \pi - \beta) \approx (0.150, 2.99)$
Let $\mcl{A}$ be the event:
    \begin{equation*}
      \mcl{A} := \bigl\{ (\theta_{3}, \ldots, \theta_{17}) \in (\beta, \pi - \beta)^{15} \bigr\} .
    \end{equation*}
The length of the interval $(\beta, \pi - \beta)$ is $\pi - 2 \beta \approx 2.84$. 
Interval length as a proportion of $2 \pi$ is $1/2 - \beta/\pi \approx 0.452$.
Let $\theta$ be a random element of $\D_{15}$, uniformly distributed on $\D_{15}$. Since $X(\theta) = 0$ when $\theta \notin \mcl{A}$, we have
    \begin{multline} \label{E:EX.a2}
      E X_{a_{2}} = E (X | \theta \in \mcl{A} ) P \{ \theta \in \mcl{A} \} 
        + E (X | \theta \notin \mcl{A} ) P \{ \theta \notin \mcl{A} \} \\
          = E (X | \theta \in \mcl{A} ) P \{ \theta \in \mcl{A} \} ,
    \end{multline}
where ``$|$'' indicates conditional expectation or probability. Now, 
    \begin{equation}  \label{E:Prob.theta.in.A}
      P \{ \theta \in \mcl{A} \} = \left( \tfrac{\pi - 2 \beta}{2 \pi} \right) ^{15} 
        \approx 6.80 \times 10^{-6}.
    \end{equation} 
Thus, to calculate $E X_{a_{2}}$, we need only estimate the conditional expectation 
$E (X_{a_{2}} | \theta \in \mcl{A} )$. We do that by simulation. It is easy to simulate the event $\theta \in \mcl{A}$: Just choose $\theta_{3}, \ldots, \theta_{17}$ randomly independently and uniformly in the interval $(\beta, \pi - \beta)$.

We simulated the event $\theta \in \mcl{A}$ two billion times. Each time we calculated $X_{a_{2}}(\theta)$. 
%
Out of the two billion simulations, $ \Theta_{a_{2}}$ occurred 2749 times. The rest of the two billion times $X_{a_{2}}$ was 0. We estimated $E (X | \theta \in \mcl{A} )$ by the sample mean, $\overline{X}_{a_{2}}$, (with 2 billion in the denominator),
of those values, which turned out to be about $1.86 \times 10^{-05}$. 
By \eqref{E:volume.is.expectation}, \eqref{E:EX.a2}, and \eqref{E:Prob.theta.in.A}
    \begin{align*}
      \Hm^{15}(\Ss_{a}) &= 2 V EX_{a} \\
        &= 2 V E (X | \theta \in \mcl{A} ) 
          P \{ \theta \in \mcl{A} \}\\
        &= 2 (2 \pi)^{15} \,  E (X | \theta \in \mcl{A} ) 
            \left( \tfrac{\pi - 2 \beta}{2 \pi} \right) ^{15} \\
        &= 2 \,  E (X | \theta \in \mcl{A} ) (\pi - 2 \beta)^{15} \\
        &\approx 237.5 . 
    \end{align*}
    
Using a less approximate value of $\Hm^{15}(\Ss_{a_{1}})$ than that given in \eqref{E:SS.a1.approx.vol} this leads to the ratio estimate  
    \begin{equation}  \label{E:measure.ratio.is.6.6.billion}
      \frac{\Hm^{15}(\Ss_{a_{1}})}{\Hm^{15}(\Ss_{a_{2}})} 
          \approx 6.6 \text{ billion} .  
    \end{equation}      
But this ratio was computed using Monte Carlo. Therefore, there is a sampling error associated with it. The standard error, the standard deviation of the sampling errors, turns out to be less than 3\% of that ratio estimate.

Suppose $\Hm^{15}(\Ss_{t})$ achieved the lower bound in 
\eqref{E:Hm.a.S.geq.R.d-p-1} for $t \in (0,17)$. If we had sharp estimates of the values of $R_{a_{1}}$ and $R_{a_{2}}$ (see proposition \ref{P:rate.of.decrease.of.Hm.St.aug.direct.mean}\eqref{I:aug.mean.essential.dist.approx}) then we could an upper bound for $\gamma$ in \eqref{E:Hm.a.S.geq.R.d-p-1} in the case of measures of location on the circle (for $n = 17$).. However, I only know $R_{t}$ asymptotically as $t \uparrow n = 17$.

\chapter{Lipschitz maps and Hausdorff Measure and Dimension}   \label{Chptr:Lip.Haus.meas.dim} 
Hausdorff dimension (Giaquinta \emph{et al} \cite[p.\ 14, Volume I]{mGgMjS98.cart.currents} and Falconer \cite[p.\ 28]{kF90}) is defined as follows. First, we define Hausdorff measure (Giaquinta \emph{et al} \cite[p.\ 13, Volume I]{mGgMjS98.cart.currents}, 
Hardt and Simon \cite[p.\ 9]{rHlS86.GMT}, and Federer \cite[2.10.2. p.\ 171]{hF69}). Let $s \geq 0$. If $s$ is an integer, let $\omega_{s}$ denote the volume of the unit ball in $\RR^{s}$:
	\begin{equation}  \label{E:vol.of.unit.ball}
		\omega_{s} = \frac{\Gamma(1/2)^{s}}{\Gamma \left( \tfrac{s}{2} + 1 \right)},
	\end{equation}
where $\Gamma$ is Euler's gamma function (Federer \cite[pp.\ 135, 251]{hF69}). If $s$ is not an integer, then $\omega_{s}$ could still be defined by \eqref{E:vol.of.unit.ball} or could be any convenient finite positive constant. Federer uses \eqref{E:vol.of.unit.ball} for any 
$s \geq 0$. Let $X$ be a metric space with metric $d_{X}$. For any subset $A$ of $X$ and $\delta > 0$ first define	
	\begin{equation}  \label{E:Hs.delta.defn}
		\Hm^{s}_{\delta}( A ) 
		      = \omega_{s} \inf \left\{ \sum_{j} 
		        \left( \frac{diam(C_{j})}{2} \right)^{s} \right\}.
	\end{equation}
Here, ``$diam$'' is diameter (w.r.t.\ $d_{X}$; see Munroe \cite[p.\ 12]{meM71.meas.thy} for definition of diameter) and the infimum is taken over all (at most) countable collections 
$\{ C_{j} \}$ of subsets of $X$ with $A \subset \bigcup_{j} C_{j}$ and 
$diam(C_{j})< \delta$. Observe that $\Hm^{s}_{\delta}(\varnothing) = 0$ since an empty cover covers $\varnothing$ and an empty sum is 0. If $X$ is second countable, it follows from Lindel\"of's theorem, Simmons \cite[Theorem A, p.\ 100]{gfS63}, 
that for any $\delta > 0$, such a countable cover exists. 
Otherwise, $\Hm^{s}_{\delta}( A ) = + \infty$.) We may assume that the covering sets 
$C_{j}$ are all open or that they are closed (Federer \cite[2.10.2, p.\ 171]{hF69}). The $s$-dimensional Hausdorff measure of $A$ is then
	\begin{equation}  \label{E:Hs.defn}
		\Hm^{s}( A ) = \lim_{\delta \downarrow 0} \Hm^{s}_{\delta}( A ) 
			= \sup_{\delta > 0} \Hm^{s}_{\delta}( A ).
	\end{equation}

$\Hm^{s}$ is ``monotonic'' and ``countably subadditive'':
    \begin{multline} \label{E:Hm.is.mono.and.count.subadditive}
    \text{If } A, B \subset X \text{ with } A \subset B \text{ then } 
      \Hm^{s}(A) \leq \Hm^{s}(B) \text{ and } \\
     \text{if } A_{1}, A_{2}, \ldots \subset X, \text{ then }
       \Hm^{s}\left( \bigcup_{i} A_{i} \right) \leq \sum_{i} \Hm^{s}(A_{i}),
         \quad s \geq 0 .
    \end{multline}

Since $\Hm^{s}_{\delta}(\varnothing) = 0$ for every $s \geq 0$ and $\delta > 0$, we have
    \begin{equation}  \label{E:Haus.measure.of.empty.set.is.0}
      \Hm^{s}( \varnothing ) = 0 \text{ for ever } s \geq 0.
    \end{equation}. 

Note that,
	\begin{multline}  \label{E:0.dim.Haus.measure}
		\Hm^{0}( A ) \text{ is the cardinality of } A \text{ if it is finite. Otherwise, } \\
		  \Hm^{0}( A ) = +\infty. \text{ In particular, } 
		    A = \varnothing \text{ if and only if } \Hm^{0}( A ) = 0.
	\end{multline}
(\emph{Pf:} Suppose $\Hm^{0}( A ) < \infty$. Then there exists $n < \infty$ s.t.\ for every $\delta > 0$ there exists a cover  $C_{1}, C_{2}, \ldots, C_{n}$ s.t.\ $diam(C_{j} < \delta$ for every $j$. Since $s$ in \eqref{E:Hs.delta.defn} is 0, we may in fact assume that each $C_{j}$ is a closed ball of diameter $\delta/2$. Then from Pollard \cite[p.\ 10]{dP90.EmpPro} we see that for every $\epsilon$ the number of pairs of points in $A$ that are more than $\epsilon$ units apart is bounded as $\epsilon \downarrow 0$. That means $A$ is finite. See Federer \cite[p.\ 171]{hF69}.)  

Suppose $Y$ is a metric space with metric $d_{Y}$, let $s \geq 0$, and let $\Hm_{Y}^{s}$ be Hausdorff measure on $Y$. Let $X \subset Y$ be non-empty. The restriction of $d_{Y}$ to $X \times X$ is a metric on $X$. Let $\Hm_{X}^{s}$ be the Hausdorff measure on $X$ computed using this metric. We \emph{claim}
    \begin{equation}  \label{E:restriction.of.Haus.meas}
      \Hm_{X}^{s}(A) = \Hm_{Y}^{s}(A) \text{ for every } A \subset X .
    \end{equation}
To see this let $A \subset X$. Let $\delta > 0$ and let $\{ C_{j} \}$ be an open cover of $A$ in $Y$. Then $\{ C_{j} \cap A \}$ is an open cover of $A$ in $X$ and
    \begin{equation*}
        \omega_{s} \sum_{j} \left( \frac{diam(C_{j})}{2} \right)^{s} 
          \geq \omega_{s} \sum_{j} \left( \frac{diam(C_{j} \cap A)}{2} \right)^{s}
            \geq \Hm^{s}_{X, \delta}( A ) ,
    \end{equation*}
where $\Hm^{s}_{X, \delta}( A )$ means just what you think it does. Taking the infimum over all covers $\{ C_{j} \}$ of $A$ in $Y$ we get $\Hm^{s}_{Y, \delta}( A ) \geq \Hm^{s}_{X, \delta}( A )$. Therefore $\Hm^{s}_{Y}( A ) \geq \Hm^{s}_{X}( A )$. Now let $\{ C_{j} \}$ be an open cover of $A$ in $X$, so $C_{j} \subset X$ for every $j$. Now, $\{ C_{j} \}$ is also an open cover of $A$ in $Y$ so
    \begin{equation*}
        \omega_{s} \sum_{j} \left( \frac{diam(C_{j})}{2} \right)^{s}  
          \geq \Hm^{s}_{Y, \delta}( A ) ,
    \end{equation*}
Taking the infimum over all covers $\{ C_{j} \}$ of $A$ in $X$ we get 
$\Hm^{s}_{X, \delta}( A ) \geq \Hm^{s}_{Y, \delta}( A )$. Therefore $\Hm^{s}_{Y}( A ) \leq \Hm^{s}_{X}( A )$ and the claim \eqref{E:restriction.of.Haus.meas} is proved.

For every $s \geq 0$, $\Hm^{s}$ is an outer measure on $X$ and, by Federer \cite[2.2.3, p.\ 61 and 2.10.2 p.\ 171]{hF69} and Hardt and Simon \cite[pp.\ 9--10]{rHlS86.GMT},
	\begin{equation}  \label{E:Hm.s.is.Borel.regular}
		\Hm^{s} \text{ is Borel regular.}
	\end{equation}
I.e., the Borel subsets of $X$ are $\Hm^{s}$-measurable and if $A \subset X$ then 
there exists a Borel $B \subset X$ s.t.\ $A \subset B$ 
and $\Hm^{s}(B) = \Hm^{s}(A)$. Note that if $X$ is a subset of a Euclidean space (and inherits the Euclidean metric) and we rescale $X$ by multiplying each vector in $X$ by $\lambda > 0$, then for every $A \subset X$ 
the measure $\Hm^{s}( A )$ will be replaced by $\lambda^{s} \, \Hm^{s}( A )$ (Falconer \cite[p.\ 27]{kF90}).

If $s$ is a positive integer, denote Lebesgue measure by $\mcl{L}^{s}$. Then by Hardt and Simon \cite[p.\ 11]{rHlS86.GMT} or theorem \ref{T:Hausdorff.measure.gives.right.value.on.manifs} below, we have 
	\begin{equation}  \label{E:when.Haus.meas.=.Leb.meas}
	     \text{If } s = 1,2,3, \ldots \text{ then } \Hm^{s} = \mcl{L}^{s} 
		  \text{ on } \RR^{s}.
	\end{equation}
(This is a special case of theorem \ref{T:Hausdorff.measure.gives.right.value.on.manifs}.)

For $A \subset X$ nonempty there will be a number $s_{0} \in [0, + \infty]$ s.t.\ 
$0 \leq s < s_{0}$ implies $\Hm^{s}( A ) = + \infty$ and $s > s_{0}$ 
implies $\Hm^{s}( A ) = 0$. That number $s_{0}$ is the ``Hausdorff dimension'', 
$\dim A$, of $A$ (Falconer \cite[p.\ 28]{kF90}). Thus, 
    \begin{equation}  \label{E:Haus.dim.defn}
      0 \leq s < \dim A \text{ implies } \Hm^{s}( A ) = + \infty 
        \text{ and } s > \dim A \text{ implies } \Hm^{s}( A ) = 0 .
    \end{equation} 
In particular, $\dim \varnothing = 0$. It is also easy to see that the dimension of a single point is 0.

In appendix \ref{Chptr:basics.of.simp.comps} we will define $\dim \sigma$, where $\sigma$ is a simplex, and $\dim P$, where $P$ is a simplicial complex. It follows from \eqref{E:when.Haus.meas.=.Leb.meas} that these dimensions are the same as the respective Hausdorff dimensions, at least if $P$ is finite. See \eqref{E:Haus.dim.of.simplex}. 

But $\Hm^{s}( A ) = 0$ is a stronger statement than $\dim A \leq s$ (Falconer \cite[p.\ 29]{kF90}). Using \eqref{E:Hm.is.mono.and.count.subadditive}, it is easy to see that 
	\begin{equation}  \label{E:dim.of.whole.=.max.dim.of.parts}
	  \text{If } A_{1}, A_{2}, \ldots \subset X,
	    \text{ then } \dim \left( \bigcup_{i} A_{i} \right)
	        = \sup \{ \dim A_{j} : j = 1, 2, \ldots \}
	\end{equation}
(Falconer \cite[p.\ 29]{kF90}). It follows that
    \begin{equation}  \label{E:dim.of.countable.set.is.0}
      \text{The dimension of a countable set is 0.}
    \end{equation}

Another way to combine spaces is by Cartesian product:
  \begin{lemma}  \label{L:dim.of.product}
Let $X$ be a non-empty metric space and let $A$ be a Lebesgue measurable subset 
of $\RR^{m}$ with $\mcl{L}^{m}(A) > 0$. ($\mcl{L}^{m} $ denotes $m$-dimensional Lebesgue measure.) If $\delta$ is a metric on $X$, define a metric $\delta$ 
on $A \times X$ as follows:
	\[
		\sigma \bigl[ (a_{1}, x_{1}), (a_{2}, x_{2}) \bigr] = \sqrt{|a_{1} - a_{2}|^{2} 
		     + \delta(x_{1},x_{2})^{2}},
		        \quad a_{1}, a_{2} \in A, \; x_{1}, x_{2} \in X.
	\]
Then w.r.t.\ $\sigma$ we have 
	\begin{equation} \label{E:dim.AxX.=.m+dimX}
		\dim (A \times X) = m + \dim X.
	\end{equation}
  \end{lemma}
  \begin{proof}
 Let $s := \dim X$. First, suppose $s = +\infty$ and let $s' \in [0, \infty)$. Since $\mcl{L}^{m}(A) > 0$, we may pick $a \in A$. Then $\{ a \} \times X \subset A \times X$ and 
	\begin{equation*}
		+ \infty \geq \Hm^{s'} \bigl( A \times X \bigr) \geq \Hm^{s'} \bigl( \{ a \} 
		  \times X \bigr) = \Hm^{s'}(X) = + \infty = m + \dim X.
	\end{equation*}
I.e., \eqref{E:dim.AxX.=.m+dimX} holds if $s = \infty$.

Next, suppose $s < \infty$. Then if $s' > s$, we have $\Hm^{s'}(X) = 0$. Therefore by Federer \cite[Theorem 2.10.45, p.\ 202]{hF69}, we have $\Hm^{m+s'}(A \times X) = 0$. Hence, 
	\begin{equation}  \label{E:dim.AxX.leq.sum}
		0 \leq \dim(A \times X) \leq m+s = m + \dim X.
	\end{equation}

Suppose $s := \dim X = 0$. And let $x \in X \neq \varnothing$. The space 
$A \times \{ x \} \subset A \times X$ is isometric to $A$, so 
$\dim(A \times X) \geq \dim (A \times \{ x \}) = \dim A =  m + \dim X$. Hence, by \eqref{E:dim.AxX.leq.sum}, we have $\dim (A \times X) = m + \dim X$. 

Now suppose $s \in (0, \infty)$ and suppose $t \in [0, s)$ and $\dim(A \times X) < m+t < m + \dim X$. Then $\Hm^{m+t}(A \times X) = 0$. But by Federer \cite[2.10.27, p.\ 190; also see statement just before 2.10.28, p.\ 191]{hF69}, this means $\mcl{L}^{m}(A) \Hm^{t}(X) = 0$. Thus, $\Hm^{t}(X) = 0$. But $t < s := \dim X$ implies $\Hm^{t}(X) = + \infty$,  contradiction.
  \end{proof}

Let $Y$ be a metric space with metric $d_{Y}$ and let $f : X \to Y$. Recall that $f$ is ``Lipschitz(ian)'' w.r.t.\ $d_{X}$ and $d_{Y}$ (Giaquinta \emph{et al} \cite[p.\ 202, Volume I]{mGgMjS98.cart.currents}, Falconer \cite[p.\ 8]{kF90}, Federer \cite[pp.\ 63 -- 64]{hF69}) if there exists $K < \infty$ (called a ``Lipschitz constant'' for $f$) s.t.\
	\[
		d_{Y} \bigl[ f(x), f(y) \bigr] \leq K \, d_{X}(x, y), 
		     \quad \text{ for every } x, y \in X.
	\]
Thus, Lipschitz maps are continuous. 
    \begin{multline}  \label{E:bi-Lipschitz.defn}
      \text{If } f \text{ is a bijection and both } f \text{ and } f^{-1} \\
        \text{ are Lipschitz, we say that } f \text{ is ``bi-Lipschitz''.}
    \end{multline}

\begin{example}   \label{Ex:dist.is.Lip}
If $S \subset X$ is nonempty then the function $y \mapsto dist(y, S) \in \RR$ is Lipschitz with Lipschitz constant 1. To see this, let $y_{1}, y_{2} \in X$. Let $\epsilon > 0$ and pick $x \in S$ s.t.\ $dist(y_{2}, S) > d_{X}(y_{2}, x) - \epsilon$. Then
	\[
		d_{X}(y_{1}, y_{2}) \geq d_{X}(y_{1}, x) - d_{X}(y_{2}, x) 
			\geq dist(y_{1}, S) - dist(y_{2}, S) - \epsilon.
	\]
Since $\epsilon > 0$ is arbitrary, we get
	\[
		d_{X}(y_{1}, y_{2}) \geq dist(y_{1}, S) - dist(y_{2}, S).
	\]
Now reverse the roles of $y_{1}$ and $y_{2}$. 
\end{example}    

Further recall the following. Let $k = 1, 2, \ldots $ and let $\mcl{L}^{k}$ denote $k$-dimensional Lebesgue measure. Suppose $T$ is a linear operator on $\RR^{k}$ and $v \in \RR^{k}$. Then by Rudin \cite[Theorems 8.26(a) and 8.28, pp.\ 173--174]{wR66.realcmplx} if $A \subset \RR^{k}$ is Borel measurable then $T(A) + v$ is Lebesgue measurable and
	\begin{equation}  \label{E:Leb.meas.of.affine.trans}
		\mcl{L}^{k} \bigl[ T(A) + v \bigr] = |\det T| \, \mcl{L}^{k}(A).
	\end{equation}
This motivates the following basic fact about Hausdorff measure and dimension
(Falconer \cite[p.\ 28]{kF90}, Hardt and Simon \cite[1.3, p.\ 11]{rHlS86.GMT}). Let $f : X \to Y$ be Lipschitz with Lipschitz constant $K$. Then for $s \geq 0$,
   \begin{equation}  \label{E:Lip.magnification.of.Hm}
      \Hm^{s} \bigl[ f(X) \bigr] \leq K^{s} \, \Hm^{s}(X). 
          \text{ Therefore, } \dim f(X) \leq \dim X,  
   \end{equation}
where in each case $\Hm^{s}$ is computed using Hausdorff measure based on the appropriate metric.

$f : X \to Y$ is ``locally Lipschitz'' (Federer \cite[p.\ 64]{hF69}) if each $x \in X$ has a neighborhood, $V$, s.t.\ the restriction $f \restriction_{V}$ is Lipschitz. So any Lipschitz map is locally Lipschitz and, conversely, 
	\begin{equation}  \label{E:local.Lip.is.Lip.on.compacts}
		\text{Any locally Lipschitz function on } X 
		  \text{ is Lipschitz on any compact subset of } X.
	\end{equation}  
Moreover,
	\begin{multline}  \label{E:comp.of.Lips.is.Lip}
	  \text{The composition of (locally) Lipschitz maps is (resp., locally) Lipschitz} \\ 
	    \text{ and the product of Lipschitz constants for the constituent functions} \\
	      \text{ is a Lipschitz constant for the composition.}
	\end{multline}

  \begin{example}[Rational functions are locally Lipschitz] \label{Ex:ratnl.fns.loc.Lip}
Here are some more examples. Let $n$ be a positive integer.

Let $x,y \in \RR \setminus \{ 0 \}$. Then $| 1/x - 1/y | = |x-y|/|xy|$. Thus, if $\delta > 0$, the function $x \mapsto 1/x$ is Lipschitz 
on $\bigl\{ x \in \RR : |x| > \delta \bigr\}$. 

If $x, y \in \RR^{n}$, then, by \eqref{E:n.c.sqrd.sum.ineq}, 
$|x + y| \leq |x| + |y| \leq \sqrt{2} \sqrt{|x|^{2} + |y|^{2}}$. It follows that addition of vectors is Lipschitz.

Let $t \in (0, \infty)$ and suppose $x_{1}, x_{2}, y_{1}, y_{2} \in [ -t, t]$. Then, by the (Cauchy-)Schwarz inequality (Stoll and Wong \cite[Theorem 3.1, p.\ 79]{rrSetW68.LinearAlgebra}),
	\begin{multline*}
		| x_{1} x_{2} - y_{1} y_{2} | = | x_{1} x_{2} - x_{2} y_{1} + x_{2} y_{1} 
		  - y_{1} y_{2} | 
		  = \bigl| (x_{2}, y_{1}) \cdot (x_{1} - y_{1}, x_{2} - y_{2}) \bigr| \\
		    \leq \bigl| (x_{2}, y_{1}) \bigr| \bigl| (x_{1} - y_{1}, x_{2} - y_{2}) \bigr| 
		      \leq \sqrt{2} \, t \bigl| (x_{1} - y_{1}, x_{2} - y_{2}) \bigr|.
	\end{multline*}
Hence, multiplication of coordinates is Lipschitz on the square $[ -t, t]^{2}$. It follows that the scalar multiplication map 
$\RR \times \RR^{n} \to \RR^{n}$ is locally Lipschitz.

Thus, by \eqref{E:comp.of.Lips.is.Lip}, polynomials, in fact, rational functions, are locally Lipschitz. 
  \end{example}

Trivally, if $x, y$ belong to a normed vector space, $|x| - |y| \leq |x-y|$. Therefore,
    \begin{equation} \label{E:norm.is.Lip}
       \text{So $x \mapsto |x|$ is Lipschitz with Lipschitz constant 1.}
    \end{equation}

  \begin{example}[Lattice operations are Lipschitz]   \label{Ex:lattice.operations.are.Lip}
Let $f(s,t) := \min(s,t)$ ($s,t \in \RR$). We show that $f$ is Lipschitz. 
Let $s,t,s',t' \in \RR$
    \begin{align*}
        \bigl| f(s,t) - f(s',t') \bigr| 
            &\leq \bigl| (s-s') 1^{\{s \leq t, s' \leq t' \}} \bigr|  
               + \bigl| (s-t') 1^{\{s \leq t, s' > t' \}} \bigr| \\
            & + \bigl| (t-s') 1^{\{s > t, s' \leq t' \}} \bigr|  
              + \bigl| (t-t') 1^{\{ s > t, s' > t' \}} \bigr| .
    \end{align*}

Suppose $s \leq t, s' > t'$. If $t' \leq s$, then $|s-t'| \leq t - t'$. If $s < t'$ then $|s-t'| \leq s' - s$. In either case,
$|s-t'| \leq |s-s'| + |t - t'|$. Similarly, if $s > t, s' \leq t'$, we have $|t-s'| \leq |s-s'| + |t - t'|$. 
Therefore, by \eqref{E:n.c.sqrd.sum.ineq}, 
    \begin{multline*}
        \bigl| f(s,t) - f(s',t') \bigr| \leq |s-s'| + \bigl( |s-s'| + |t - t'| \bigr) \\
          + \bigl( |s-s'| + |t - t'| \bigr) + |t - t'| \leq 3 \sqrt{2} \, \bigl| (s,t) - (s',t') \bigr| .
    \end{multline*}

$\max$ is similar. (Or use $\max(s,t) = - \min(-s, -t)$.)
  \end{example}

  \begin{example}[Infimum and supremum are Lipschitz]   \label{Ex:inf.sup.are.Lip}
Let $S$ be a non-empty set and let $X$ be a set of bounded real-valued functions on $S$. Metrize $X$ by the sup norm $\| \cdot \|$. Let $f : X \to \RR$ be the supremum functional. Let $\alpha, \beta \in X$ and let $\delta := \| \alpha - \beta| \|$. Pick $x \in S$ s.t.\ 
$\alpha(x) + \delta \geq f(\alpha)$. Then $\beta(x) + 2 \delta \geq f(\alpha)$. Thus, 
$f(\beta) + 2 \delta \geq f(\alpha)$. I.e., $f(\alpha) - f(\beta) \leq 2 \| \alpha - \beta \|$. Reversing the roles of $\alpha$ and $\beta$ we get $\bigl| f(\alpha) - f(\beta) \bigr| \leq 2 \| \alpha - \beta \|$. Thus, the supremum functional is Lipschitz w.r.t.\ the sup norm and the Euclidean norm on $\RR$. Since $\inf \alpha = - \sup (- \alpha)$, it follows that the infimum functional is Lipschitz as well.
  \end{example}
  
  \begin{example}[Projection is Lipschitz]   \label{Ex:projection.is.Lip}
   Let $X$ and $Y$ be metric spaces with metrics $d_{X}$ and $d_{Y}$, resp. Let $d_{XY}$ be a metric on $X \times Y$ s.t.\ 
    \begin{equation*}
      d_{XY} \bigl[ (x_{1}, y_{1}), (x_{2}, y_{2}) \bigr] 
       \geq \max\bigl[ d_{X}(x_{1}, x_{2}), d_{Y}(y_{1}, y_{2}) \bigr] ,
         \quad x_{1}, x_{2} \in X \text{ and } y_{1}, y_{2} \in Y .
    \end{equation*}
E.g., $d_{XY} \bigl[ (x_{1}, y_{1}), (x_{2}, y_{2}) \bigr] := 
   \sqrt{ d_{X}(x_{1}, x_{2})^{2} + d_{Y}(y_{1}, y_{2})^{2} }$ or
$d_{XY} \bigl[ (x_{1}, y_{1}), (x_{2}, y_{2}) \bigr] 
:= d_{X}(x_{1}, x_{2}) + d_{Y}(y_{1}, y_{2})$. Then trivially, projection $(x,y) \mapsto x$ is Lipschitz with Lipschitz constant 1.
  \end{example}
	
Note that, since $\Hm^{0}$ is just cardinality for finite sets and $+ \infty$ for infinite sets, whether $f : X \to Y$ is locally Lipschitz or not, we have
	\begin{equation}  \label{E:maps.reduce.H0}
		\Hm^{0} \bigl[ f(X) \bigr] \leq \Hm^{0}(X).
	\end{equation}
An easy consequence of \eqref{E:Lip.magnification.of.Hm} is the following.

  \begin{lemma}  \label{L:loc.Lip.image.of.null.set.is.null}
Let $X$ and $Y$ be metric spaces with $X$ separable. Suppose $f : X \to Y$ is locally Lipschitz. If $s \geq 0$ and $\Hm^{s}(X) = 0$, then $\Hm^{s} \bigl[ f(X) \bigr] = 0$. In particular, $\dim f(X) \leq \dim X$.
  \end{lemma}
  \begin{proof}
By Lindel\"of's theorem (Simmons \cite[Theorem A, p.\ 100]{gfS63}) $X$ can be partitioned into a countable number of disjoint Borel sets $A_{1}, A_{2}, \ldots$ on each of which $f$ is Lipschitz with respective Lipschitz constant $K_{i}$. By  \eqref{E:Lip.magnification.of.Hm}, we have
	\begin{equation*}
		\Hm^{s} \bigl[ f(X) \bigr] \leq \sum_{i} \Hm^{s} \bigl[ f(A_{i}) \bigr] 
			\leq \sum_{i} K_{i}^{s} \Hm^{s}(A_{i}).
	\end{equation*}
  \end{proof}
  
Another generalization of \eqref{E:Lip.magnification.of.Hm} is the following. (See Boothby \cite[Theorem (2.2), p.\ 26]{wmB75}.)

	\begin{lemma}   \label{L:loc.Lip,avg.bound.on.Haus.meas}
	Let $k$ and $m$ be positive integers. Let $U \subset \RR^{k}$ be open and let $M$ be an $m$-dimensional Riemannian manifold with Riemannian metric tensor $x \mapsto \langle \cdot, \cdot \rangle_{x}$. Suppose 
$h = (h_{1}, \ldots, h_{m}) : U \to M$ is continuously differentiable. 
Then $h$ is locally Lipschitz on $U$ w.r.t.\  Euclidean distance and the topological metric corresponding to $\langle \cdot, \cdot \rangle$. For $y \in U$, define the matrix
   \begin{equation*}
      \Omega(y)^{k \times k} 
       := \left( \left\langle h_{\ast} \left( \frac{\partial}{\partial z_{i}} \restriction_{z = y} \right), 
       h_{\ast} \left( \frac{\partial}
         {\partial z_{j}} \restriction_{z = y} \right) \right\rangle_{h(y)} \right).
   \end{equation*}
At each $x \in U$, let $\lambda(x)^{2}$ be the largest eigenvalue of $\Omega(x)$
(with $\lambda(x) \geq 0$). Then $\lambda$ is continuous. 

Furthermore, let $a \geq 0$ and let $A \subset U$ be Borel with $\Hm^{a}(A) < \infty$.  Then 
	   \begin{equation}  \label{E:lwr.bnd.on.meas.of.image}
	      \Hm^{a} \bigl[ h(A) \bigr]  \leq \int_{A} \lambda(x)^{a} \, \Hm^{a}(dx).
	   \end{equation}
	\end{lemma}

\begin{proof} 
(See Boothby \cite[Theorem (2.2), pp.\ 26--27]{wmB75}.) By lemma \ref{L:Eigen.cont.} and continuity of $Dh$, $\lambda$ is continuous. Let $\epsilon > 0$. 
Since $\lambda$ is continuous, by Lindel\"of's theorem 
(Simmons \cite[Theorem A, p.\ 100]{gfS63}), 
there exists an at most countable cover, $C_{1}, C_{2}, \ldots $, of $U$ by open convex sets with the property
   \begin{equation}   \label{E:narrow.rnge.lambda.on.Ci}
         x, x' \in C_{i} \Rightarrow \bigl| \lambda(x)^{a} - \lambda(x')^{a} \bigr| < \epsilon, 
            \quad (i = 1, 2, \ldots).
   \end{equation}

For each $i = 1, 2, \ldots$ let $\Lambda_{i} = \sup_{x \in C_{i}} \lambda(x)$. 
We prove the \emph{claim:} on each $C_{i}$, the function $h$ is Lipschitz with Lipschitz constant 
$\Lambda_{i}$. In particular, $h$ is locally Lipschitz on $U$.
(See Giaquinta \emph{et al} \cite[Theorem 2, p.\ 202, Vol. I]{mGgMjS98.cart.currents}.)  
Let $x,y \in C_{i}$. Think of $x,y$ as row vectors. Since $C_{i}$ is open and convex there is an open interval $I \supset [0,1]$ s.t.\ for every $u \in I$ we have $\ell(u) := (1-u)x + u y \in C_{i}$.
The function $f := h \circ \ell : I \to M$ is defined and differentiable. It defines an arc in $M$. Let $\rho$ be the topological metric corresponding to the Riemannian metric 
$\langle \cdot, \cdot \rangle$. Let $\| X \| := \sqrt{ \langle X, X \rangle_{x} }$ for $X \in T_{x} (M)$, $x \in M$. Then, by \eqref{E:alpha.length.integral},
   \begin{align}   \label{E:bound.incr.by.arc.lngth}
      \rho \bigl[ h(y), h(x) \bigr] 
        &\leq \text{length of arc } f \notag \\
        &= \int_{I} \, \bigl\| f_{\ast} (d/du) \bigr\|_{f(u)} \, du \notag  \\
        &= \int_{I} \, \bigl\| h_{\ast} \circ \ell_{\ast} (d/du) \bigr\|_{f(u)} \, du  \\
        &= \int_{I} \, \left\| h_{\ast} \left[ \sum_{i=1}^{k} (y_{i} - x_{i}) 
           \frac{\partial}{\partial z_{i}} \restriction_{z = \ell(u)} \right] 
             \right\|_{f(u)} \, du  \notag \\
        &= \int_{I} \, \sqrt{ (y-x) \Omega \bigl[ \ell(u) \bigr] (y-x)^{T} } \, du. \notag
   \end{align}
   
Let $u \in [0,1]$ and let $w = \ell(u) \in C_{i} \subset U \subset \RR^{k}$. Let 
$\lambda^{2}_{1} \geq \lambda^{2}_{2} \geq \cdots \geq \lambda^{2}_{k} \geq 0$ be the eigenvalues of $\Omega(w)$, so $\lambda^{2}(w) = \lambda^{2}_{1}$. 
Let $z_{1}, \ldots, z_{k} \in \RR^{k}$ be corresponding orthonormal eigenvectors, thought of as row vectors. Write
$y - x = \sum_{j=1}^{k} \alpha_{j} z_{j}$. Then
	\begin{align*}
		(y - x)  \Omega(w) (y - x)^{T}  
		  &= \left( \sum_{j=1}^{k} \alpha_{j} z_{j} \right) \Omega(w) 
		            \left( \sum_{j=1}^{k} \alpha_{j} z_{j}^{T} \right)   \\
		  &= \left( \sum_{j=1}^{k} \alpha_{j} z_{j} \right) 
		            \left( \sum_{j=1}^{k} \alpha_{j} \lambda^{2}_{j} z_{j}^{T} \right)   \\
		  &= \sum_{j=1}^{k} \lambda^{2}_{j} \alpha_{j}^{2}  \\
		  &\leq \lambda^{2}(w) \sum_{j=1}^{k} \alpha_{j}^{2} \\
		  &\leq \Lambda_{i}^{2} | y - x |^{2}.
	\end{align*}
I.e., 
	\[
		\sqrt{ (y-x) \Omega (w) (y-x)^{T} } \leq \Lambda_{i} | y - x |.
	\]
Substituting this into \eqref{E:bound.incr.by.arc.lngth} and noting that $I$ has measure less than 1 proves the claim that on each $C_{i}$, the function $h$ is Lipschitz with Lipschitz constant $\Lambda_{i}$.

Let $A_{1} = A \cap C_{1}$. Having defined $A_{1}, \ldots, A_{n}$, let
   \[
      A_{n+1} = (A \cap C_{n+1}) \setminus \left( \bigcup_{i=1}^{n} A_{i} \right).
   \]
Then $A_{1}, A_{2}, \ldots$ is a Borel partition of $A$. By \eqref{E:Lip.magnification.of.Hm} and \eqref{E:narrow.rnge.lambda.on.Ci},
   \[
      \Hm^{a} \bigl[ h(A) \bigr] \leq \sum_{i} \Hm^{a} \bigl[ h(A_{i}) \bigr] 
         \leq \sum_{i} \Lambda_{i}^{a} \Hm^{a} (A_{i}) 
         \leq \int_{A} \lambda(x)^{a} \Hm^{a}(dx) + \Hm(A) \epsilon.
   \]
Since $\epsilon > 0$ is arbitrary and $\Hm^{a}(A) < \infty$, the lemma follows.
\end{proof}

\begin{remark}  \label{R:when.M.=.R.m}
 Suppose in the preceding that $M = \RR^{m}$. Then 
	\[
		\Omega(x) = Dh(x)^{T} Dh(x), \text{ where }
			Dh(x) := \left( \frac{\partial h_{i}(y)}{\partial y_{j}} \right)_{y = x}.
	\]
\end{remark}

Regarding the following, see Boothby \cite[Theorem (2.2), p.\ 26]{wmB75}.

  \begin{lemma}   \label{L:coord.maps.are.Lip}
Let $M$ be a Riemannian manifold, of dimension $m < \infty$. 
Let, $\| \cdot \|_{M,x'}$ is the norm on $T_{x'} M$ corresponding to the Riemannian metric at $x' \in M$. Let $\delta$ be the topological metric on $M$ corresponding 
to $\| \cdot \|_{M,\cdot}$. Let $x \in M$ and let $(\clU_{0}, \varphi)$ be a coordinate neighborhood on $M$ with $x \in \clU_{0}$. 
We may assume $\varphi(x) = 0 \in \RR^{m}$.
Let $G_{0} := \varphi(\clU_{0}) \subset \RR^{m}$ and let $\psi : G_{0} \to \clU_{0}$ be the inverse of $\varphi$. For $x' \in \clU_{0}$, let 
$E_{ix'} :=  \psi_{\ast} \left( \tfrac{\partial}{\partial z^{i}} 
\restriction_{z = \varphi(x')}  \right)$ ($i = 1, \ldots, m$) be the coordinate frame field on $\clU_{0}$ at $x'$ and let 
$\Gamma_{M,x'}^{m \times m}$ be the matrix of the Riemannian tensor at $x'$ 
w.r.t.\ $E_{1x'}, \ldots, E_{mx'}$. Suppose $\clU$ is an open neighborhood of $x$ 
with $\clU \subset \overline{\clU} \subset \clU_{0}$ and $\overline{\clU}$ is compact. 
   \begin{enumerate}
	\item For $x' \in \clU$, let $\mu_{1}(x') \geq \cdots \geq \mu_{m}(x') > 0$ be the eigenvalues of $\Gamma_{M,x'}$. 
	Then there exists $\blds{\mu} \in [1, \infty)$ s.t.\
	\begin{equation}  \label{E:Gam.eigvals.bdd.away.from.0.infty}
		\mu_{1}(x') \leq \blds{\mu}^{2} \text{ and } \, 1/ \mu_{m}(x') 
		        \leq \blds{\mu}^{2} \text{ for every } x' \in \clU.
	\end{equation}
		\label{I:bold.mu.bounds.eigvals}
	\item Let $\blds{\mu}$ be any number as in part \ref{I:bold.mu.bounds.eigvals}. There exists $r_{0} > 0$ s.t.\ $\overline{B_{r_{0}}(0)} \subset G := \varphi(\clU)$. There exists an open neighborhood $H \subset G$ s.t. $\overline{H} \subset B_{r_{0}}(0)$ and having the following property. Let $y_{1}, y_{2} \in H$, so the line segment joining $y_{1}$ and $y_{2}$ lies entirely in $B_{r_{0}}(0) \subset G$. Let $\Delta = |y_{2} - y_{1}|$. Define the linear arc $\xi : \bigl[ 0, \Delta \bigr] \to H \subset B_{r_{0}}(0)$ joining $y_{1}$ and $y_{2}$ defined by 
	\[
		\xi(s) = y_{1} + \frac{s}{ \Delta } ( y_{2} - y_{1} ), \quad 0 \leq s \leq \Delta.
	\]
Let $\alpha = \psi \circ \xi$, so $\alpha$ is a curve in $\clU$ joining $\psi(y_{1})$ and $\psi(y_{2})$. Then 
	\begin{equation*}
		\blds{\mu}^{-1} \leq \| \alpha'(s) \|_{M,\alpha(s)} \leq \blds{\mu}
	\end{equation*}
for every $s \in (0, \Delta)$. 
		\label{I:alpha'.shorter.longer.mu}
	\item Let $\blds{\mu}$, etc., be as in part \ref{I:alpha'.shorter.longer.mu}. 
	Then $\blds{\mu}$ is a Lipschitz constant for the restrictions $\psi \restriction_{H}$ 
	and $\varphi \restriction_{\psi(H)}$ w.r.t.\ the Euclidean metric on $G$ and $\delta$. 
	   \label{I:MU.is.Lip.cnst.for.phi.psi}
	\item  Let $\bigl\{ (\clU_{0\gamma}, \varphi_{0\gamma}) : \gamma \in C \bigr\}$ be a covering of $M$ by coordinate neighborhoods. Then $M$ has a covering 
	$(\clU_{1}, \varphi_{1}), (\clU_{2}, \varphi_{2}), \ldots$ by coordinate neighborhoods s.t.\ each $\overline{\clU}_{i}$ is a subset of some $\clU_{0\gamma}$ and for 
the same $\gamma$ the coordinate map $\varphi_{i}$ is the restriction of $\varphi_{0\gamma}$ to $\clU_{i}$. We may also assume that each $\clU_{i}$ has compact closure. Each coordinate map $\varphi_{i}$ is Lipschitz and has a Lipschitz inverse (in both cases w.r.t.\ the Euclidean metric on $\RR^{m}$ and the topological metric corresponding to the Riemannian tensor on $M$). For each $i = 1, 2, \ldots$, let $\blds{\mu}_{i} \in (1, \infty)$ be any number s.t.\ part \ref{I:bold.mu.bounds.eigvals} above holds with $\clU = \clU_{i}$ and $\blds{\mu} = \blds{\mu}_{i}$. Then $\blds{\mu}_{i}$ is a Lipschitz constant for both $\varphi_{i}$ and its inverse. In particular, if $M$ is compact, then $\Hm^{m}(M) < \infty$. In particular, $M$ is a ``Lipschitz manifold'' (section \ref{SS:asymp.prob}).   \label{I:phi.psi.Lip}
   \end{enumerate}
  \end{lemma}

Call the coordinate neighborhoods $(\clU_{1}, \varphi_{1}), (\clU_{2}, \varphi_{2}), \ldots$ as in part \ref{I:phi.psi.Lip} ``bi-Lipschitz'' coordinate neighborhoods. 

\begin{proof} 
  
If $x' \in M$, let $\langle \cdot, \cdot \rangle_{M,x'}$ be the Riemannian 2-form (``metric'') on $M$ at $x'$. Let $\| \cdot \|_{M,x'}$ be the corresponding norm. Let $\delta$ be the topological metric on $\clU_{1}$ determined by the Riemannian metric, $\langle \cdot, \cdot \rangle_{M}$ (Boothby \cite[Theorem (3.1), p.\ 187]{wmB75}).  

The 2-tensor $\langle \cdot, \cdot \rangle_{M,x'}$ has a symmetric positive definite matrix, $\Gamma_{M,x'}^{m \times m}$, w.r.t.\ the local coordinate frame field $E_{ix'} :=  \psi_{\ast} \left( \tfrac{\partial}{\partial z^{i}} \restriction_{z = \varphi(x')}  \right)$ ($x' \in \clU_{0}$; $i = 1, \ldots, m$). 
Thus, the $i,j^{th}$ entry in $\Gamma_{M,x'}$ is $\langle E_{ix'}, E_{jx'} \rangle_{M, x'}$. By Boothby \cite[Definition (2.1), p.\ 182]{wmB75}, the entries of the $r \times r$ matrix $\bigl( [ X_{y'i}, X_{y'j} ]_{y'} \bigr)$ are continuous in $y' \in \mbf{A}$, 
the entries in the matrix $\Gamma_{M,x'}$ are continuous in $x' \in \clU_{0}$. Let $\mu_{1}(x') \geq \ldots \geq \mu_{m}(x') > 0$ be the eigenvalues of $\Gamma_{M,x'}$. By lemma \ref{L:Eigen.cont.}, the eigenvalues $\mu_{1}(x'), \ldots, \mu_{m}(x')$ are continuous in $x' \in \clU_{0}$. Therefore, by compactness of $\overline{\clU}$, there exists a number $\blds{\mu} \in [1,  \infty)$ s.t.\ \eqref{E:Gam.eigvals.bdd.away.from.0.infty} holds.

Let $m := \dim M$. Let $x \in M$ and let $\varphi : \clU_{0} \to \RR^{m}$ be a coordinate neighborhood of $x$. 
By proposition \ref{P:geod.cnvx.nbhds.exist} there are neighborhoods $\clU_{1}$ and $\clU$ of $x$ s.t.\  
$\clU \subset \overline{\clU_{1}} \subset \clU_{0}$, $\overline{\clU_{1}}$ is compact, and $\clU$ is geodesically convex. 
Let $G_{0} = \varphi(\clU_{0}) \subset \RR^{m}$ and let $\psi : G_{0} \to \clU_{0}$ be the inverse of $\varphi$. We may assume $\varphi(x) = 0$. 

If $x' \in \clU_{0}$, the differential, $\varphi_{\ast}$, of $\varphi$ maps the tangent space $T_{x'} M$ onto $T_{\varphi(x')} \RR^{m}$. For $y \in G_{0}$, a basis for $T_{\varphi(x')}$ is $\tfrac{\partial}{\partial z_{i}} \restriction_{z=y}$ ($i=1, \ldots, m$) (Boothby \cite[Corollary (1.5), p.\ 109]{wmB75}). 
Let $\psi_{\ast} : T \RR^{m} \to T M$ be the differential of $\psi$ (Boothby \cite[Remark (1.3), p.\ 108]{wmB75}).

For $\eta > 0$, let 
	\begin{multline*}
		\mcl{B}_{\eta}(x) := \{ x' \in \clU_{0} : 
		   \text{There is a geodesic arc connecting } x' \text{ and } x  \\
		      \text{ and at least one such arc has length } < \eta  \}.
	\end{multline*}
By Boothby \cite[Theorem (3.1), p.\ 187]{wmB75}, $\mcl{B}_{\eta}(x)$ is an open neighborhood of $x$. Pick $\eta_{0} = \eta_{0}(x) > 0$ so small that 
$\mcl{B}_{\eta_{0}}(x) \subset \clU$. In particular, $\overline{ \mcl{B}_{\eta_{0}}(x) }$ is compact. By proposition \ref{P:geod.cnvx.nbhds.exist}, we may assume 
$\mcl{B}_{\eta_{0}}(x)$ is geodesically convex. Let $G = \varphi(\clU)$ and 
$H = \varphi \bigl( \mcl{B}_{\eta_{0}}(x) \bigr)$. Choose $r_{0} \in (0, \infty)$ so small that the closure of the ball
	\[
		B_{r_{0}}(0) := \bigl\{ y \in \RR^{m} : | y | < r_{0} \bigr\}
	\]
lies in $G$. By making $\eta_{0}$ smaller if necessary, we may assume 
$\overline{H} \subset B_{r_{0}}(0)$.

Let $x_{1}, x_{2} \in \mcl{B}_{\eta_{0}}(x)$ 
and let $y_{i} = \varphi(x_{i}) \in H \subset B_{r_{0}}(0) \subset G$ ($i=1,2$). Let $\gamma : [0, \lambda] \to M$ be the unique shortest geodesic 
in $M$ joining $x_{1}$ and $x_{2}$, which exists by geodesic convexity of $\mcl{B}_{\eta_{0}}(x)$ and its image lies in $\mcl{B}_{\eta_{0}}(x)$. We may assume that 
$\gamma$ is parametrized by arclength. In particular, by proposition \ref{P:geod.cnvx.nbhds.exist},
	\begin{equation}   \label{E:lambda.=.delta.x1.x2}
		\lambda = \delta(x_{1}, x_{2}).
	\end{equation}

Let $\omega := \varphi \circ \gamma$, so $\omega : [0, \lambda] \to G$ joins $y_{1}$ and $y_{2}$. We can extend $\omega$ to a slightly larger, open interval $J \supset [0, \lambda]$ 
s.t.\ $\omega : J \to G$ is differentiable.\footnote{$\gamma$ is a geodesic. That means that for some vector $X \in T_{x_{1}}$ with $\| X \|_{M, x_{1}} = 1$ we have $\gamma(t) = Exp_{x_{1}}(t X)$ ($0 \leq t \leq \lambda$). For some $\epsilon > 0$, we may extend this to $\gamma(t) = Exp_{x_{1}}(t X)$ ($-\epsilon < t < \lambda + \epsilon$). $\omega := \varphi \circ \gamma$ is automatically extended.}
Write $\omega(t) = \bigl( \omega^{1}(t), \ldots, \omega^{m}(t) \bigr) \in \RR^{m}$ ($t \in J$). Thus, 
	\[
		\omega'(t) := \bigl( (\omega^{1})'(t), \ldots, (\omega^{m})'(t) \bigr) \in \RR^{m}.
	\]
Here, $(\omega^{i})'(t)$ are just numbers and we regard 
$\omega'(t) = \bigl( (\omega^{1})'(t), \ldots, (\omega^{d})'(t) \bigr)$ as a row vector. By contrast, write 
$\gamma'(t) = \gamma_{\ast} \bigl[ (d/du)_{u=t} \bigr] \in T_{\gamma(t)} M$  (Boothby \cite[Theorem (1.2), p.\ 107]{wmB75}). Hence, by Boothby \cite[Theorem (1.6), p.\ 109]{wmB75},
	\[
		\omega_{\ast}\bigl[ (d/du)_{u=t} \bigr] 
		  = \sum_{i=1}^{m} (\omega^{i})'(t) \, 
		    \left( \tfrac{\partial}{\partial z_{i}} \restriction_{z=\omega(t)} \right).
	\]

Now, $\gamma = \psi \circ \omega$ and, by definition of $E_{ix'}$ ($i = 1, \ldots, m$), we have 	   
	\begin{equation}   \label{E:gamma'.in.trms.omega.E}
		\gamma'(t) 
		 = \psi_{\ast} \circ \omega_{\ast}\bigl[ (d/du)_{u=t} \bigr] 
		  = \psi_{\ast} \left[ \sum_{i=1}^{m} (\omega^{i})'(t) 
		          \left( \tfrac{\partial}{\partial z_{i}} \restriction_{z=\omega(t)} \right) \right]
		   = \sum_{i=1}^{m} (\omega^{i})'(t) \, E_{i \gamma(t) }.
         \end{equation}
Hence, since $\gamma$ is parametrized by arclength,
	\begin{equation*}
		1 = \| \gamma'(t) \|_{M,\gamma(t)}^{2} 
		         = \bigl\langle \gamma'(t), \gamma'(t) \bigr\rangle_{M,\gamma(t)}(x') 
		        = \omega'(t) \, \Gamma_{M,\gamma(t)} \, \omega'(t)^{T} 
		     \geq \mu_{m} \bigl[ \gamma(t) \bigr] \bigl| \omega'(t) \bigr|^{2}.
	\end{equation*}  
Therefore, by \eqref{E:Gam.eigvals.bdd.away.from.0.infty} we have 
	\begin{equation}  \label{E:omega.lngth.bdd}
		| \omega'(t) | \leq \blds{\mu} < \infty \text{ for every } t \in [0, \lambda].
	\end{equation}

Now, $|y_{2} - y_{1}|$ is no greater than the length of the curve $\omega$ in $G$. But by \eqref{E:omega.lngth.bdd}, Boothby \cite[p.\ 185]{wmB75}, and \eqref{E:lambda.=.delta.x1.x2} that length is 
	\[
		\int_{0}^{\lambda} \bigl| \omega'(t) \bigr| \, dt \leq \lambda \blds{\mu} 
		         = \blds{\mu} \, \delta(x_{1}, x_{2}).
	\] 
Thus, $|y_{2} - y_{1}| \leq \blds{\mu} \, \delta(x_{1}, x_{2})$. This proves that $\varphi$ is Lipschitz with Lipschitz constant $\leq$ $\blds{\mu}$.

Let $\Delta = |y_{2} - y_{1}|$. Define the linear arc $\xi : \bigl[ 0, \Delta \bigr] \to B_{r_{0}}(0)$ joining $y_{1}$ and $y_{2}$ defined by 
	\[
		\xi(s) = y_{1} + \frac{s}{ \Delta } ( y_{2} - y_{1} ), \quad 0 \leq s \leq \Delta.
	\]
Thus, by Boothby \cite[Theorem (1.6), p.\ 109]{wmB75} again,
	\begin{equation}  \label{E:xi.ast.Delta.yi}
		\xi^{\ast} \bigl[ (d/du)_{u=t} \bigr]  
		     = \Delta^{-1} \sum_{i=1}^{m} ( y_{2}^{i}-y_{1}^{i} )
		          \left( \tfrac{\partial}{\partial z_{i}} \restriction_{z=\xi(t)} \right),
	\end{equation}
where $y_{j}^{i}$ is the $i^{th}$ coordinate of $y_{j} \in H \subset \RR^{m}$. The length of $\xi$ is $\Delta$, of course. Let $\alpha = \psi \circ \xi$, so $\alpha$ is a curve in $M$ joining $x_{1}$ and $x_{2}$. Let $\alpha'(t) := \alpha^{\ast} \bigl[ (d/du)_{u=t} \bigr] \in T_{\alpha(t)} M$, ($0 \leq t \leq \Delta$). Thus, by \eqref{E:xi.ast.Delta.yi}, we have as in \eqref{E:gamma'.in.trms.omega.E}
	\begin{equation}   \label{E:alpha'.in.trms.Delta.y.E}
		\alpha'(t) = \Delta^{-1} \sum_{i=1}^{m} ( y_{2}^{i}-y_{1}^{i} ) \, E_{i \alpha(t)}.
	\end{equation}

The distance $\lambda = \delta(x_{1}, x_{2})$ between $x_{1}$ and $x_{2}$ is no greater than the length of $\alpha$. The length of $\alpha$ is
	\begin{equation}  \label{E:length.of.alpha}
		\ell(\alpha) = \int_{0}^{\Delta} \| \alpha'(s) \|_{M,\alpha(s)} \, ds.
	  \end{equation}
But, by definition of $\Gamma_{M,\cdot}$,  
	\begin{equation*}  
              \| \alpha'(s) \|_{M,\alpha(s)}^{2} 
		    = \Delta^{-2} (y_{2} - y_{1}) \, \Gamma_{M,\alpha(s)}  
		             \, (y_{2} - y_{1})^{T}
	\end{equation*}
(regarding $y_{2} - y_{1} \in \RR^{m}$ as a row vector). Therefore, by \eqref{E:Gam.eigvals.bdd.away.from.0.infty},
	\begin{equation}   \label{E:bound.on.alpha'}
	        \blds{\mu}^{-1} \leq \sqrt{\mu_{m}} 
	          = \Delta^{-1} |y_{2} - y_{1}| \, \sqrt{\mu_{m}} \leq
		\| \alpha'(s) \|_{M,\alpha(s)} 
		  \leq \Delta^{-1} |y_{2} - y_{1}| \, \sqrt{\mu_{1}} = \sqrt{\mu_{1}}
		      \leq \blds{\mu},
	\end{equation}
since $|y_{2} - y_{1}| = \Delta$. This establishes statement \ref{I:alpha'.shorter.longer.mu}. Continuing, we see, by \eqref{E:length.of.alpha} and \eqref{E:bound.on.alpha'}, 
	\[
		\delta(x_{1}, x_{2}) \leq \ell(\alpha) \leq \int_{0}^{\Delta} \blds{\mu} \, ds 
			= \blds{\mu} \Delta = \blds{\mu} |y_{2} - y_{1}|.
	\]
Since $x_{i} = \psi(y_{i})$, this proves that $\psi$ is also Lipschitz with Lipschitz constant no greater than $\blds{\mu}$. This completes the proof of statement \ref{I:MU.is.Lip.cnst.for.phi.psi}.

Since $x \in M$ we started with is arbitrary and $M$ is second countable (Munkres \cite[Definition 1.1, p.\ 3]{jrM66}), by Lindel\"of's theorem, Simmons \cite[Theorem A, p.\ 100]{gfS63}, there is a countable collection $( \clU_{i}, \varphi_{i} )$ with the properties described in statement \ref{I:phi.psi.Lip} of the lemma.

If $M$ is compact, then it is covered by finitely many $\clU_{i}$. Since $\clU_{i}$ is relatively compact, $\varphi_{i}( \clU_{i} )$ is bounded. Therefore, $\mcl{L}^{m} \bigl[ \varphi_{i}( \clU_{i} ) \bigr] < \infty$. Hence, by \eqref{E:when.Haus.meas.=.Leb.meas} and \eqref{E:Lip.magnification.of.Hm}, we have $\Hm^{m}( \clU_{i} ) < \infty$. Thus, $\Hm^{m}(M) < \infty$. This completes the proof of statement \ref{I:phi.psi.Lip} of the lemma.
\end{proof}

The following useful fact follows from lemma \ref{L:coord.maps.are.Lip}, lemma \ref{L:loc.Lip,avg.bound.on.Haus.meas}, and \eqref{E:comp.of.Lips.is.Lip}.

  \begin{corly}  \label{C:cont.diff.=.loc.Lip}
Let $M$ and $N$ be a Riemannian manifolds of dimension $m$ and $n$, resp.
Let $h : M \to N$ be continuously differentiable. Then $h$ is locally Lipschitz with respect to the topological metrics induced by the Riemannian metrics on $M$ and $N$. In particular, if $A \subset M$ is compact then $h$ is Lipschitz on $A$. In particular, if $M$ has two Riemannian metrics, then the identity map on $M$ is locally Lipschitz w.r.t.\ the topological metrics induced by the two Riemannian metrics.
  \end{corly}
  
Here is a more precise version in a special case. (See lemma \ref{L:imbedding.is.loc.Lip}.)
  
It is easy to deduce the following from what we have proved so far
(especially \eqref{E:when.Haus.meas.=.Leb.meas}, \eqref{E:dim.of.whole.=.max.dim.of.parts}, and lemma \ref{L:coord.maps.are.Lip}; use Boothby \cite[Theorem (4.5), p.\ 193]{wmB75}) or theorem \ref{T:Hausdorff.measure.gives.right.value.on.manifs}. (Or see Falconer \cite[p.\ 29]{kF90}.)
   \begin{equation}  \label{E:Haus.dim.s-manif.=.s} 
    \text{The Hausdorff dimension of an $s$-dimensional differentiable manifold is } s .
   \end{equation} 

The following gives more specific information. It generalizes \eqref{E:when.Haus.meas.=.Leb.meas}. Use ``$\mres$'' to denote restriction of measures to a set. See Boothby \cite[Section V.7, pp.\ 213--217]{wmB75} for discussion of volume on a Riemannian manifold. 

  \begin{theorem} \label{T:Hausdorff.measure.gives.right.value.on.manifs}
[Giaquinta \emph{et al} \cite[Theorem 1, p.\ 15, Volume 1]{mGgMjS98.cart.currents}] 
 The Hausdorff measure $\Hm^{n}$ in $\RR^{n}$ coincides with the Lebesgue measure 
 $\mcl{L}^{n}$ on $\RR^{n}$. Moreover, if $M$ is an $s$-dimensional smooth submanifold 
 of $\RR^{n}$, $0 < s < n$, then $\Hm^{s} \mres M$ is the standard volume measure in $M$ induced by the Euclidean metric in $M$. 
  \end{theorem}
(See \eqref{E:when.Haus.meas.=.Leb.meas}.)

  \begin{lemma}  \label{L:metric.dominance.means.Haus.meas.dominance}
Let $X$ be a second countable, locally compact topological space and let $\rho_{1}$ and $\rho_{2}$ be two metrics on $X$ generating the topology. Suppose inclusion $i : (X, \rho_{1}) \to (X, \rho_{2})$ is locally Lipschitz. I.e., $X$ is covered by open sets $U$ with the following property. There exists $K(U) < \infty$ (depending on $U$) s.t.\ for every $x,y \in U$ we have
	\begin{equation}  \label{E:rho1.rho2.locally.equivalent}
		\rho_{2}(x,y) \leq K(U) \, \rho_{1}(x,y).
	\end{equation}  
Let $s \geq 0$ and let $\Hm^{i,s}$ be $s$-dimensional Hausdorff measure w.r.t.\ $\rho_{i}$ ($i=1,2$). Then there exists a locally bounded Borel measurable function $M : X \to [0, +\infty)$ (i.e., every $x \in X$ has an open neighborhood on which $M$ is bounded) s.t.\ for every $\Hm^{s}$-measurable set $A \subset X$ we have
	\begin{equation}  \label{E:Hm2.leq.int.M.Hm1}
		\Hm^{2,s}(A) \leq \int_{A} M(x) \, \Hm^{1,s}(dx).
	\end{equation}
If $A \subset U$, where $U \subset X$ is one of the sets in the open cover, then $\Hm^{2,s}(A) \leq K(U)^{s} \Hm^{1,s}(A)$.
  \end{lemma}

\begin{proof} 
By Lindel\"of's theorem (Simmons \cite[Theorem A, p.\ 100]{gfS63}) and Ash \cite[Theorem A5.15, p.\ 387]{rbA72}, there exist open sets $A_{1}, A_{2}, \ldots$ s.t.\ $A_{n} \uparrow X$ and for each $n$, the closure $\overline{A_{n}}$ is compact. Hence, for $n = 1, 2, \ldots$, the restriction $i \restriction_{A_{n}} : \bigl( A_{n}, \rho_{1} \restriction_{A_{n} \times A_{n}} \bigr) \to (X, \rho_{2})$ is Lipschitz. If $x \in A_{1}$, define $M'(x)$ to be the Lipschitz constant, $M_{1} \in [0, \infty)$, corresponding to $A_{1}$. If $n > 1$ and $x \in A_{n} \setminus A_{n-1}$, let $M'(x)$ be the Lipschitz constant, $M_{n} \in [0, \infty)$, corresponding to $A_{n}$. The function $M$ defined in this way is clearly Borel and locally bounded. 

Now let $A \subset X$ be $\Hm^{s}$-measurable set $A \subset X$. By \eqref{E:Lip.magnification.of.Hm},
we have
	\begin{align*}
		\Hm^{2,s}(A) &= \Hm^{2,s}(A \cap A_{1}) 
		      + \sum_{n \geq 2} \Hm^{2,s}(A \cap A_{n} \setminus A_{n-1})    \\
		                           &\leq M_{1,s} \, \Hm^{1,s}(A \cap A_{1}) 
		      + \sum_{n \geq 2} M_{n,s} \, \Hm^{1,s}(A \cap A_{n} \setminus A_{n-1})    \\
		                           &= \int_{A} \bigl[ M'(x) \bigr]^{s} \, \Hm^{1,s}(dx).
	\end{align*}
Let $M := (M')^{s}$.

The last sentence of the corollary is immediate from \eqref{E:Lip.magnification.of.Hm}.
\end{proof}

A circumstance in which \eqref{E:rho1.rho2.locally.equivalent} holds is described in the following. 

  \begin{lemma} \label{L:imbedding.is.loc.Lip}
Let $M$ and $N$ be Riemannian manifolds and let $f : N \to M$ be an imbedding. 
Let $\mu$ and $\nu$ be the topological metrics on $M$ and $N$, resp., determined by the given Riemannian metrics on the two manifolds. Define a metric $f^{\ast} \mu$ 
on $N$ by $f^{\ast} \mu(x,y) := \mu \bigl[ f(x), f(y) \bigr]$ ($x,y \in N$).
Then the identity map $N \to N$ is locally Lipschitz w.r.t.\ to $f^{\ast} \mu$ and $\nu$ and also w.r.t.\ $\nu$ and $f^{\ast} \mu$.
 \end{lemma}

Note that $f^{\ast} \mu$ is \emph{not} necessarily the same as the metric on $N$ determined by the pullback under $f^{\ast}$ of the Riemannian metric on $M$. \emph{Example:} Consider a round sphere, $S^{k} \subset \RR^{k+1}$. Then the straight line distance between two points in $S^{k}$ is not the same as the great circle distance, which is the geodesic distance between them w.r.t.\ the Riemannian metric induced on $S^{k}$ by the inclusion $S^{k} \hookrightarrow \RR^{k+1}$.

  \begin{proof}  
Let $x \in N$. Since $f$ is an imbedding, $f^{-1}$ is $C^{\infty}$ 
(Boothby \cite[Theorem (5.5), p.\ 78]{wmB75}). 
Hence, by corollary \ref{C:cont.diff.=.loc.Lip}, both $f$ and $f^{-1}$ are locally Lipschitz. Thus, if $V \subset N$ is a relatively compact neighborhood of $x$ then $f$ is Lipschitz on $V$ and $f^{-1}$ is Lipschitz on $U := f(V)$. Let $w, z \in U$. 
Then there exists $K < \infty$ s.t.\  
	\begin{equation*}
		\nu(w,z) = \nu \bigl[  f^{-1} \circ f(w),  f^{-1} \circ f(z) \bigr] 
		  \leq K \mu \bigl[  f(w),  f(z) \bigr] = K f^{\ast} \mu(w,z).
	\end{equation*}
Similarly, there exists $K < \infty$ s.t.\  
	\begin{equation*}
		f^{\ast} \mu(w,z) = \mu \bigl[  f(w),  f(z) \bigr] \leq K \nu(w,z).
	\end{equation*}
The lemma is proved.
  \end{proof}

\chapter{Simplicial Complexes}  \label{Chptr:basics.of.simp.comps}
This appendix presents some of the material in Munkres \cite{jrM84}, mostly from pages 2 -- 11, 83, and 371 plus a general theorem from \cite{spE.PolyhedralApprox} and a strengthening of the latter. (See also Rourke and Sanderson \cite{cpRbjS72.PiecewiseLinearTopol}.)   Let $N$ be a positive integer and let $n \in \{0, \ldots, N \}$. Points $v(0), \ldots, v(n)$ in $\RR^{N}$ are ``geometrically independent'' (or are in ``general position'') if $v(1) - v(0), \ldots, v(n) - v(0)$ are linearly independent. Equivalently, $v(0), \ldots, v(n)$ are geometrically independent if and only if
	\[
		\sum_{i=1}^{n} t_{i} = 0 \text{ and } \sum_{i=1}^{n} t_{i} v(i) = 0
	\]  
together imply $t_{0} = \cdots t_{n} = 0$. If $v(0), \ldots, v(n) \in \RR^{N}$ are geometrically independent then they are the vertices of the ``simplex''
	\[
	     \sigma = \{ \beta_{0} v(0) + \cdots + \beta_{n} v(n) : 
		  \beta_{0}, \ldots \beta_{n} \geq 0 
	         \text{ and } \beta_{0} + \cdots + \beta_{n} = 1 \}.
	\]
We say that $\sigma$ is ``spanned'' by $v(0), \ldots, v(n)$ and $n$ is the ``dimension'' of $\sigma$. We write $\sigma = \langle v(0), \ldots, v(n) \rangle$. (Sometimes we call $\sigma$ a ``$n$-simplex'' and write $\dim \sigma = n$. 
We adopt the convention that the statement ``$\sigma$ is spanned by $v(0), \ldots, v(n)$'' or a reference to $\langle v(0), \ldots, v(n) \rangle$ implies that $v(0), \ldots, v(n)$ are geometrically independent. Note that $\sigma$ is convex and compact. Indeed, it is the convex hull of $\{ v(0), \ldots, v(n) \}$. Thus, every $y \in \sigma$ can be expressed uniquely (and continuously) in ``barycentric coordinates'':
   \[
      y = \sum_{v \text{ is a vertex in } \sigma} \beta_{v}(y) \, v,
   \]
where the $\beta_{v}(y)$'s are nonnegative and sum to 1. We have
	\begin{multline}  \label{E:simplex.determines.verts}
		\text{Given a simplex, } \sigma, \text{ there exists one and only 
		    one} \\ 
		       \text{geometrically independent set of points spanning } \sigma.
	\end{multline}
We have
    \begin{equation*}
      y = \sum_{i=0}^{n} \beta_{v(i)} v(i) 
        = \sum_{i=1}^{n} \beta_{v(i)} v(i) 
          + \left( 1 -  \sum_{i=1}^{n} \beta_{v(i)} \right) v(0)
            = v(0) + \sum_{i=1}^{n} \beta_{v(i)} \bigl( v(i) - v(0) \bigr) .
    \end{equation*}
Recall \eqref{E:bi-Lipschitz.defn}. Since $v(0), v(1), \ldots, v(n)$ are geometrically independent it follows that 
    \begin{equation}  \label{E:barycen.coords.on.simp.are.bi-Lip}
      \text{The map } y \mapsto (\beta_{v(0)}, \ldots, \beta_{v(n))} \text{ is bi-Lipschitz.}
    \end{equation}
(See proposition \ref{P:bary.coords.are.Lip} below.) 

The simplex $\sigma$ lies on the plane
	\begin{equation}  \label{E:formla.for.affine.plane}
	   \Pi = \{ \beta_{0} v(0) + \cdots + \beta_{n} v(n) : \beta_{0} 
		  + \cdots + \beta_{n} = 1 \}.
	\end{equation}
I.e., the definition of $\Pi$ is like that of $\sigma$ except the non-negativity requirement is dropped. Note that $\Pi$ need not include the origin of $\RR^{N}$. $\Pi$ is the smallest plane containing 
$\sigma$. The dimension of $\Pi$ is $n$. Any nonempty subset of $\{ v(0), \ldots, v(n) \}$ is geometrically independent and the simplex spanned by that subset is a ``face'' 
of $\sigma$. So $\sigma$ is a face of itself and a vertex of $\sigma$ is also a face of $\sigma$. A ``proper'' face of $\sigma$ is a face of $\sigma$ different from $\sigma$. If $\tau$ is a proper face of $\sigma$, write $\sigma \succ \tau$. 

Let $\sigma$ be a simplex spanned by geometrically independent points 
$v(0), \ldots, v(n)$. If $J \subsetneqq \{ 0, \ldots, n \}$ is nonempty, let $\tau$ be  the proper face of $\sigma$ spanned by $\{ v(j), \, j \in J \}$. E.g., $\tau$ might be a vertex 
of $\sigma$. The face ``opposite'' $\tau$ is the span, $\omega$, 
of $\{ v(j), \, j \notin J \}$ (Munkres \cite[p.\ 5 and Exercise 4, p.\ 7]{jrM84}). Thus, 
$\tau$ consists of those $y \in \sigma$ s.t.\ $\beta_{v}(y) = 0$ for all vertices 
$v \notin \tau$ and $\omega$ consists of those $y \in \sigma$ s.t.\ $\beta_{v}(y) = 0$ for all vertices $v \in \tau$. 

The union of all proper faces of $\sigma$ is the ``boundary'' of $\sigma$, denoted 
$\text{Bd} \, \sigma$. The ``(simplicial) interior'' of $\sigma$ (as a simplex) is the set 
$\text{Int} \, \sigma := \sigma \setminus (\text{Bd} \, \sigma)$, where ``$\setminus$'' indicates set-theoretic subtraction. 
	\begin{equation}  \label{E:criterion.for.int.simp}
		\text{If  } y = \sum_{j=0}^{n} \beta_{j}(y) \, v(j), 
		      \text{ then } y \in \text{Int} \, \sigma \text{ if and only if } \beta_{j}(y) > 0 
		          \text{ for all }  j = 0, \ldots, n.
	\end{equation}
Thus, the interior of $\sigma$ as a simplex is in general different from its (usually empty) interior as a subspace of $\RR^{N}$. In fact, the interior (as a simplex) of a 0-dimensional simplex (a single point) is the point itself. But $\sigma$ is the topological closure of 
$\text{Int} \, \sigma$ and $\text{Int} \, \sigma$ is the relative interior of $\sigma$ as a subset of $\Pi$ defined by \eqref{E:formla.for.affine.plane}.

By \eqref{E:dim.of.countable.set.is.0}, the Hausdorff dimension (see \eqref{E:Haus.dim.defn}) of a point is 0 so if $\sigma$ is a 0-dimensional simplex then 
$\dim \sigma = 0$. Inductively, let $n = 0, 1, \ldots$ and suppose that if $m \leq n$ then any $m$-dimensional simplex has Hausdorff dimension $m$. 
Let $\sigma$ be an $(n+1)$-dimensional simplex with $n \geq 0$. By \eqref{E:criterion.for.int.simp} and \eqref{E:barycen.coords.on.simp.are.bi-Lip}, barycentric coordinates provide a bi-Lipschitz homeomorphism of $\text{Int} \, \sigma$ onto an open subset of the $(n+1)$-dimensional affine plane $\beta_{0} + \cdots + \beta_{n+1} = 1$, i.e., onto a $(n+1)$-dimensional smooth manifold. Therefore, by lemma \ref{L:loc.Lip.image.of.null.set.is.null} and \eqref{E:Haus.dim.s-manif.=.s}, 
$\dim (\text{Int} \, \sigma) = n+1$. By the induction hypothesis and \eqref{E:dim.of.whole.=.max.dim.of.parts}, 
$\dim \sigma = \max \left( \{\dim (\text{Int} \, \sigma) \} 
\cup \{ \dim \tau : \tau \prec \sigma \} \right) = n+1$, where ``$\prec$'' means ``is a proper face of''. This proves
    \begin{equation}  \label{E:Haus.dim.of.simplex}
      \text{The Hausdorff dimension of an $n$-dimensional simplex is } n. 
    \end{equation}
 
Recall $I_{n}$ is the $n$-dimensional identity matrix;  $1^{n}$ is $n$-dimensional column vector of 1's.
	\begin{lemma}  \label{L:simp.lin.ineqs}
Let $N > n > 0$ be integers and let $\sigma \subset \RR^{N}$ be an $n$-dimensional simplex.
Then there exist matrices $A^{N \times n}$ and $B^{N \times (N-n)}$  of rank $n$ and $N-n$, resp., satisfying $B^{T} A = 0$, and there exist row vectors $y \in \RR^{n}$ and $z \in \RR^{N-n}$ s.t.\  $x^{1 \times N} \in \sigma$ if and only if 
	\begin{equation}  \label{E:cellular.exprssn.for.simplex}
		x A \, (I_{n}, - 1^{n}) \geq (y, -y 1^{n} - 1)^{1 \times (n+1)} \text{ and } x B = z, 
	\end{equation}
where inequalities of vectors are defined coordinate-wise. Conversely, given $(A,B,y,z)$ as above, the set of $x \in \RR^{N}$ satisfying \eqref{E:cellular.exprssn.for.simplex} is an $n$-simplex. We have $x \in \text{Int} \, \sigma$ if and only 
the inequalities in \eqref{E:cellular.exprssn.for.simplex} are strict.

If $n = 0$ then we have a similar result except we do not need $A$: We ignore the inequalities in \eqref{E:cellular.exprssn.for.simplex} involving $A$. Similarly, if $n = N$ we  ignore the equalities involving $B$. If $n < N$ then the lowest dimensional affine plane on which $\sigma$ lies is $\{ x \in \RR^{N} : x B = z \}$.

In particular, simplices are ``cells'' in the sense of Munkres \cite[Definition 7.2, p.\ 71]{jrM66}. 	
	\end{lemma}
	\begin{proof} 
 First, let $n=0$ so $N-n=N$. Let $\sigma$ be a 0-simplex with single vertex 
$v_{0} \in \RR^{N}$. Then $x \in \sigma$ if and only if
$x B = z$, with $B = I_{N}$ ($N$-dimensional identity matrix) and $z = v_{0}$.

Now let $n > 0$ and suppose $x$ lies in an $n$-simplex, $\sigma \subset \RR^{N}$, spanned by $v_{0}, \ldots, v_{n}$. 
Let $V^{N \times n}$ be the matrix whose $j^{th}$ column is $(v_{j} - v_{0})^{T}$ ($j \in \NN_{n}$). Then $x^{1 \times N} \in \sigma$ if and only if  
	\begin{equation}  \label{E:simplex.condns.on.beta}
		x = v_{0} + \beta V^{T},  \text{ where }
		   \beta \text{ is any row } n\text{-vector s.t.\ } \beta \geq 0 
		      \text{ and } \beta 1^{n} \leq 1.
	\end{equation}
Note that, by \eqref{E:criterion.for.int.simp}, $x \in \text{Int} \, \sigma$ if and only if the inequalities in the preceding are strict. Now, $v_{0}, \ldots, v_{n}$ are implicitly geometrically independent. (See \eqref{E:simplex.determines.verts}. In particular, $n \leq N$.) Hence, $V$ has rank $n$. Therefore, $A^{N \times n} := V (V^{T} V)^{-1}$ has rank $n$. 

Suppose $n = N$. Then $V$ is a square matrix and is invertible. Hence, by \eqref{E:simplex.condns.on.beta}, $x \in \sigma$ if and only if $x A = v_{0} A + \beta^{1 \times n}$ (``if'' because $A$ is invertible) with $\beta \geq 0$ and $\beta 1^{n} \leq 1$. Let $y := v_{0} A$. Then, by \eqref{E:simplex.condns.on.beta}, $x \in \sigma$ if and only if 
	\begin{equation}  \label{E:xA.ineq}
	     x A \, (I_{n}, - 1^{n}) = (xA, - xA 1^{n}) = (y + \beta, -y 1^{n} - \beta 1^{n}) 
		  \geq (y, -y 1^{n} - 1).
	\end{equation}
Thus, $x \in \sigma$ if and only if \eqref{E:cellular.exprssn.for.simplex} holds, ignoring the part about ``$B$''.

Now suppose, $0 < n < N$. First, assume $x$ lies in an $n$-simplex, $\sigma$ and let $V$ and $A$ be as above. Let $B^{N \times (N-n)}$  be any matrix of rank $N-n$ s.t.\ $B^{T} V = 0$. Then $B^{T} A = 0$ and, by \eqref{E:simplex.condns.on.beta} again,  
	\begin{multline*}  
                xA = v_{0}A + \beta \text{ and } x B = (v_{0} + \beta V^{T}) B  = v_{0} B, \\
                \text{ where } \beta \text{ is a row } n\text{-vector s.t.\ } \beta \geq 0
		      \text{ and } \beta 1^{n} \leq 1.
	\end{multline*}
Letting $y := v_{0}A$ and $z := v_{0} B$, \eqref{E:cellular.exprssn.for.simplex} follows as in \eqref{E:xA.ineq}. 
Moreover, if $x \in \text{Int} \, \sigma$ then the inequalities in \eqref{E:cellular.exprssn.for.simplex} are clearly strict. 

Conversely, let $S \subset \RR^{N}$ be the set of $x^{1 \times N}$ s.t.\ \eqref{E:cellular.exprssn.for.simplex} holds. 
Let $x^{1 \times N} \in \RR^{N}$ and $\beta^{n \times 1} := xA - y$. Let $Z^{N \times N} := (A, B)$. Since $A$ and $B$ are each of full rank and $B^{T} A = 0$, 
we have $rank \, Z = N$ so $Z$ is invertible. We have 
	\begin{equation}   \label{E:eks.zee}
		x \in S \text{ if and only if } xZ = (y + \beta, z), \beta \geq 0, \text{ and } \beta 1^{n} \leq 1. 
	\end{equation}
Write  
	\begin{equation*}
		Z^{-1} =
		  \begin{pmatrix}
			V^{n \times N} \\
			W^{(N-n) \times N}
		  \end{pmatrix} 
		  .
	\end{equation*}
So $V$ and $W$ have rank $n$ and $N-n$, resp. Then \eqref{E:eks.zee} holds if and only if $x = (y+\beta, z) Z^{-1} = (yV +zW) + \beta V$, 
$\beta \geq 0$, and $\beta 1^{n} \leq 1$. I.e., if and only if $x$ lies in the $n$-simplex, $\sigma$, with vertices $v_{0}, v_{1}, \ldots, v_{n}$, where $v_{0}^{1 \times N} = yV +zW$ and $(v_{i} - v_{0})^{T}$ is the $i^{th}$ row of $V$ ($i \in \NN_{n}$). Thus, $S = \sigma$.
Moreover, $x \in \text{Int} \, \sigma$ if and only if the $n+1$ inequalities are strict \eqref{E:criterion.for.int.simp}. 

Let $P := \{ x \in \RR^{N} : x B = z \}$. It is immediate, by \eqref{E:cellular.exprssn.for.simplex}, that $\sigma \subset P$. 
Since $B$ has rank $N-n$, $\dim P = n$. But since $\dim \sigma = n$ the lowest dimensional affine plane containing $\sigma$ has dimension $n$. So that plane must be $P$.
	\end{proof}

The following lemma (Munkres \cite[lemma 1.1, p.\ 6]{jrM84}) about convex sets is handy. 

\begin{lemma}  \label{L:rays.intersect.bndry.in.one.pt}
Let $U$ be a bounded, convex, open set in some affine space (e.g., a Euclidean space). Let $w \in U$. Then each ray emanating from $w$ intersects the boundary of $U$ in precisely one point.
\end{lemma}

Let $v(0), \ldots, v(n) \in \RR^{n}$ be the vertices of $\sigma$. Then
	\begin{equation}  \label{E:barycenter.of.sigma}
      		\hat{\sigma} := \frac{1}{n+1} \sum_{j=0}^{n} v(j) \in \text{Int} \, \sigma
	\end{equation}
is the ``barycenter'' of $\sigma$ (Munkres \cite[p.\ 85]{jrM84}). By \eqref{E:criterion.for.int.simp}, 
	\begin{equation}  \label{E:bary.in.Int}
		\hat{\sigma} \in \text{Int} \, \sigma.
	\end{equation} 
Munkres \cite[p.\ 90]{jrM66} defines the 
``radius'', $r(\sigma)$, of $\sigma$ to be the minimum distance from $\hat{\sigma}$ to $\text{Bd} \, \sigma$. He defines the ``thickness'' of the simplex $\sigma$ to be
$t(\sigma) := r(\sigma)/diam(\sigma)$. Here, ``$diam(\sigma)$'' is the diameter of $\sigma$, i.e., the length of the longest edge of $\sigma$. Thus, e.g., the thickness of any 1-simplex is 1/2. 

A ``simplicial complex'', $P$, in $\RR^{N}$ is a collection of simplices in $\RR^{N}$ s.t.\
	\begin{equation}   \label{E:complex.contains.all.faces}
		\text{Every face of a simplex in $P$ is in $P$.}
	\end{equation}
and		     
	\begin{equation}   \label{E:intersection.of.simps}
		\text{The intersection of any two simplices in $P$ is a face of each of them.}  
	\end{equation}
It turns out that an equivalent definition of simplicial complex is obtained by replacing condition \eqref{E:intersection.of.simps} by the following. 
	\begin{equation}  \tag{\ref{E:intersection.of.simps}'}
		\text{Every pair of distinct simplices in $P$ have disjoint interiors.}
	\end{equation}  
It follows that
	\begin{multline}   \label{E:Int.rho.cuts.sigma.then.rho.in.sigma}
		\text{If } \rho, \sigma \text{ are elements of a simplicial complex and } 
		       (\text{Int} \, \sigma) \cap \rho \ne \varnothing  \\
		       \text{ then } \sigma \text{ is a face of } \rho.
	\end{multline}
(\emph{Proof:} $(\text{Int} \, \sigma) \cap \rho$ lies in some face of $\rho$. Let $\tau$ be the smallest face of $\rho$ (in terms of inclusion) containing 
$(\text{Int} \, \sigma) \cap \rho$. ($\tau = \rho$ is possible.)  Suppose $\sigma \ne \tau$. Then by (\ref{E:intersection.of.simps}') $(\text{Int} \, \sigma) \cap \tau$ lies in some proper face of $\tau$. But $(\text{Int} \, \sigma) \cap \tau  = \bigl[ (\text{Int} \, \sigma) \cap \rho \bigr] \cap \tau = (\text{Int} \, \sigma) \cap \rho$, since $(\text{Int} \, \sigma) \cap \rho \subset \tau \subset \rho$. I.e., $(\text{Int} \, \sigma) \cap \rho$ lies in a proper face of $\tau$. That contradicts the minimality of $\tau$. Therefore, $\sigma =  \tau$.)

A simplicial complex, $P$, is ``finite'' if it is finite as a set (of simplices). The 
    \begin{multline}  \label{E:simplicial.dim.of.cmplx}
      \text{``dimension'' of a simplicial complex is the supremum } \\
        \text{ of the set of dimensions of the constituent simplices.}
    \end{multline}
(See Munkres \cite[p.\ 14]{jrM84}. Thus, infinite dimensional simplicial complexes are possible.) In the following assume $P$ is a non-empty simplicial complex. We will use 
``$\dim P$'' to denote the dimension of a simplicial complex $P$. As we shall see presently (see \eqref{E:Haus.dim.of.cmplx}), this usage does not conflict with that defined in \eqref{E:Haus.dim.defn}. 

A subset, $L$, of $P$ is a ``subcomplex'' of $P$ if $L$ is a simplicial complex in its own right. The collection, $P^{(q)}$, of all simplices in $P$ of dimension at most $q \geq 0$ is a subcomplex, called the ``$q$-skeleton'' of $P$. In particular, $P^{(0)}$ is the set of all vertices of simplices in $P$. The ``polytope'' or ``underlying space'' of $P$, denoted by $|P|$, is just the union of the simplices in $P$. If $P$ is finite, i.e., consists of finitely many simplices, then $|P|$ is assigned the relative topology it inherits from $\RR^{N}$. In general, a subset, $X$, of $|P|$ is closed (open) if and only if 
$X \cap \sigma$ is closed (resp.\ open) in $\sigma$ for every $\sigma \in P$. Call this topology the ``polytope topology'' on $|P|$. A space that equals $|P|$ for some simplicial complex, $P$, is called a ``polyhedron''. By Munkres \cite[Lemma 2.5, p.\ 10]{jrM84}, 
	\begin{equation}  \label{E:|P|.compact}
		|P| \text{ is compact if and only if } P \text{ is finite.}
	\end{equation}  
If $X$ is a topological space, then a 
    \begin{multline}  \label{E:triangulation.defn}
      \text{``triangulation'' of $X$ is a simplicial complex, } P, \\
        \text{ and a homeomorphism } f : |P| \to X .
    \end{multline}
 (See Munkres \cite[p.\ 118]{jrM84}).
 
 Recall \eqref{E:simplicial.dim.of.cmplx}. By \eqref{E:dim.of.whole.=.max.dim.of.parts} and \eqref{E:Haus.dim.of.simplex}, 
    \begin{equation}  \label{E:Haus.dim.of.cmplx}
       \dim P = \dim |P| . 
    \end{equation}

Let $P$ be a finite simplicial complex of positive dimension. As in Munkres \cite[p.\ 10]{jrM84}, define ``barycentric coordinates'' on $|P|$ as follows. First, note that
	\begin{equation} \label{E:x.in.exctly.1.simplex.intrr}
		\text{If $x \in |P|$ then there is exactly one simplex 
		    $\tau \in P$ s.t.\ $x \in \text{Int} \, \tau$.}
	\end{equation}
(To see this, note that since $P$ is finite, there is a smallest simplex (w.r.t.\ inclusion order), $\tau$, in $P$ containing $x$. Clearly, $x \in \text{Int} \, \tau$. By (\ref{E:intersection.of.simps}') this implies $\tau$ is unique.)  Let $\tau^{(0)}$ be the set of vertices of $\tau$. Then, by \eqref{E:criterion.for.int.simp}, there exist strictly positive numbers 
$\beta_{v}(x)$ ($v \in \tau^{(0)}$) that sum to 1 and satisfy
	\[
		x = \sum_{v \in \tau^{(0)}} \beta_{v}(x) v.
	\]
Since $v \in \tau ^{(0)}$ are geometrically independent, the coefficients $\beta_{v}(x)$, $v \in \tau ^{(0)}$, are unique. 
If $v \in P^{(0)}$ is not a vertex of $\tau$ define $\beta_{v}(x) = 0$. Thus, 
	\[
		x = \sum_{v \in P^{(0)}} \beta_{v}(x) v, \quad \quad x \in |P|.
	\]
The entries in $\bigl\{ \beta_{v}(x), \; v \in P^{(0)} \bigr\}$ are the ``barycentric coordinates'' of $x$. The barycentric coordinates of $x \in |P|$ are unique. For if not, then $x$ lies in the interiors of each of two distinct simplices in $P$. This contradicts (\ref{E:intersection.of.simps}'). For each $v \in P^{(0)}$ the function $\beta_{v}$ is continuous on $|P|$ (Munkres \cite[p.\ 10]{jrM84}). Recall \eqref{E:barycen.coords.on.simp.are.bi-Lip}. If $P$ is finite, we have the following. 

	\begin{prop}  \label{P:bary.coords.are.Lip}
Let $P$ be a finite simplicial complex. Then the vector-valued function 
$\blds{\beta} : x \mapsto \bigl\{ \beta_{v}(x), \; v \in P^{(0)} \bigr\}$ is Lipschitz in $x \in |P|$ (w.r.t.\ the obvious Euclidean metrics; see appendix \ref{Chptr:Lip.Haus.meas.dim}).
	\end{prop}
In the course of proving this lemma, the following useful fact emerges.

\begin{corly}   \label{C:reverse.triangle.ineq.in.simp.cmplxs}
Let $P$ be a finite simplicial complex. There exists $K < \infty$, depending only on $P$, s.t.\ the following holds. Let $\rho, \tau \in P$ satisfy $\rho \cap \tau \ne \varnothing$, but suppose neither simplex is a subset of the other. If $x \in \text{Int} \, \rho$ and $y \in \text{Int} \, \tau$ then there exist 
$\tilde{x}, \tilde{y} \in  \text{Int} \, (\rho \cap \tau)$ s.t.\ 
	\[
		|x - \tilde{x}| + |\tilde{x} - \tilde{y}| + |\tilde{y} - y| \leq K |x-y|.
	\]
\end{corly}

\begin{proof}[Proof of proposition \ref{P:bary.coords.are.Lip}]
Let $x, y \in |P|$.   Since $P$ is a finite complex there exists $\delta_{1} > 0$ s.t.\ if $\rho, \tau \in P$ are disjoint then $dist(\rho, \tau) > 2 \delta_{1}$. 
Let $\rho$ ($\tau$) be the unique simplex in $P$ s.t.\ $x \in \text{Int} \, \rho$ (respectively [resp.], $y \in \text{Int} \, \tau$; see \eqref{E:x.in.exctly.1.simplex.intrr}). Therefore, if $\rho$ and $\tau$ are disjoint then the Euclidean length $|x-y|$ is bounded below by $2 \delta_{1}$. Moreover, $\bigl| \blds{\beta}(z) \bigr| \leq 1$ for every $z \in |P|$ since the components 
of $\blds{\beta}(x)$ are nonnegative and sum to 1. Thus,
	\begin{equation}  \label{E:x.y.more.than.delta1.apart}
		\bigl| \blds{\beta}(x) - \blds{\beta}(y) \bigr| \leq (1/\delta_{1}) |x - y| 
		    \text{ if $x$ and $y$ lie in disjoint simplicies.}
	\end{equation}

So assume $\rho \cap \tau \ne \varnothing$. In fact, first consider the behavior of $\beta$ on a single simplex, $\rho$ in $P$. (This covers the case where $\tau\subset \rho$ or \emph{vice versa}.) Suppose $\rho$ is an $n$-simplex, so $\rho$ has $n+1$ vertices $v(0), \ldots, v(n)$. If $n=0$, i.e., $\rho$ is a single point, then $\beta$ is trivially Lipschitz on $\rho$. So suppose $n > 0$. We show that $\beta$ is Lipschitz on $\rho$.  
We can assume $|P| \subset \RR^{N}$ for some $N \geq n$. 
Let $V^{(n+1) \times N}$ be the matrix whose $i^{th}$ row is $v(i-1)$ ($i \in \NN_{n}+1$). (Use superscripts to indicate matrix dimension.)  Let $V_{0}^{n \times N}$ be the matrix whose $i^{th}$ row is $v(i) - v(0)$ ($i \in \NN_{n}$). Recall 
$(1^{n})^{n \times 1}$ is the column vector $(1, \ldots, 1)^{T}$, \eqref{E:1n.col.vec.defn}. Thus,
	\begin{equation}  \label{E:V0.from.V}
		(-1^{n} \; \;  I_{n} ) V = V_{0},
	\end{equation}	
where $I_{n}$ is the $n \times n$ identity matrix.

The vertices of $\rho$ are geometrically independent so $V_{0}$ has full rank $n$. This means $V_{0} V_{0}^{T}$ is invertible. But by \eqref{E:V0.from.V} 
$(-1^{n} \; \;  I_{n} ) V V_{0}^{T} = V_{0} V_{0}^{T}$. Therefore, 
$W^{(n+1) \times n} := V V_{0}^{T}$ has rank $n$. This implies 
that the vector $(1^{n+1})^{(n+1) \times 1} = (1, \ldots, 1)^{T}$ is not in the column space 
of $W^{(n+1) \times n}$. For suppose for some column vector $\alpha$ we have 
$W \alpha = 1^{n+1}$. Then $\alpha \ne 0$ and from \eqref{E:V0.from.V} and the fact that $V_{0} V_{0}^{T}$ is nonsingular we have
	\[
		0 \ne V_{0} V_{0}^{T} \alpha = (-1^{n} \; I_{n} ) W \alpha 
		      = (-1^{n} \; I_{n} ) 1^{n+1} = 0.
	\]
Therefore, $(W, \; 1^{n+1})$ is invertible.

For $x \in \rho$, let $\bigl( \blds{\beta}^{\rho}(x) \bigr)^{1 \times (n+1)}$ be the row vector 
$\bigl( \beta_{v(0)}(x), \ldots, \beta_{v(n)}(x) \bigr)$. Think of $x \in \RR^{N}$ as a row vector. Then we have $x = \blds{\beta}^{\rho}(x) V$ and 
$1 = \blds{\beta}^{\rho}(x) 1^{n+1}$. Therefore,
	\[
		(x V_{0}^{T}, \; 1) = \blds{\beta}^{\rho}(x) (W, \; 1^{n+1})^{(n+1) \times (n+1)}.
	\]
But we have just observed that $U^{(n+1) \times (n+1)} := (W, \; 1^{n+1})$ is invertible. Therefore,
	\[
		\blds{\beta}^{\rho}(x) = (x V_{0}^{T}, \; 1) U^{-1}.
	\]
Hence, $\blds{\beta}^{\rho}$ is affine on $\rho$. Therefore, $\blds{\beta}^{\rho}$ and, hence, $\blds{\beta}$ is Lipschitz on $\rho$. Since $P$ is a finite complex there is $K < \infty$ that works as a Lipschitz constant for every simplex in $P$. I.e.,
	\begin{equation}  \label{E:beta.unif.Lip.on.each.simplex}
		\bigl| \blds{\beta}(x) - \blds{\beta}(x')  \bigr| \leq K |x - x'| 
		\text{ for every } x, x' \in \rho \text{ for every } \rho \in P.
	\end{equation}

It remains to tackle the case 
	\begin{multline}  \label{E:rho.intersect.but.aren't.subsets}
		x \in \text{Int} \, \rho \text{ and } y \in \text{Int} \, \tau ; \; \rho, \tau \in P; \;  \\
		     \rho \cap \tau \ne \varnothing \text{ but } \rho \text{ is not a subset of } 
		       \tau \text{ and } \tau \text{ is not a subset of } \rho.
	\end{multline}
$\rho \cap \tau \ne \varnothing$ but $\rho$ is not a subset of $\tau$ and $\tau$ is not a subset of $\rho$. In this case, by \eqref{E:Int.rho.cuts.sigma.then.rho.in.sigma},  
$(\text{Int} \, \rho) \cap (\text{Int} \, \tau) = \varnothing$. We handle this case by reducing it to the last case. By \eqref{E:intersection.of.simps}, $\rho \cap \tau$ is a simplex, a proper face of both $\rho$ and $\tau$. Let $\xi$ be the face of $\rho$ opposite $\rho \cap \tau$ and let $\omega$ be the face of $\tau$ opposite $\rho \cap \tau$. Let $x \in \text{Int} \, \rho$ and $y \in \text{Int} \, \tau$. 

\emph{Claim:} There is a unique $z_{0} = z_{0}(x) \in \xi$ s.t.\ the line passing through $x$ and $z_{0}$ intersects $ \text{Int} \, (\rho \cap \tau)$. Given $z \in \xi$, the line, $L(z) = L(z,x)$, passing through $z$ and $x$ is unique since $x \in \text{Int} \, \rho$ implies $x \notin \xi$. Let $v(0), \ldots, v(n)$ be the vertices of $\rho$ and, renumbering if necessary, we may assume $v(0), \ldots, v(m)$ are the vertices of $\rho \cap \tau$ for some $m = 0, \ldots, n-1$. Then $v(m+1), \ldots, v(n)$ are the vertices of $\xi$. Let $z \in \xi$ and write
	\[
		z = \sum_{i=m+1}^{n} \mu_{i} v(i),
	\]
where the $\mu_{i}$'s are nonnegative and sum to 1.

First, we prove there is at most one $z \in \xi$ s.t.\ 
$L(z) \cap \rho \cap \tau \ne \varnothing$. Suppose 
$L(z)$ intersects $\rho \cap \tau$ at $\tilde{x} = \sum_{i=0}^{m} \mu_{i} v(i)$. Then for some $t \in \RR$ with $t \ne 1$ we have
	\begin{multline}  \label{E:xtilde.is.convx.combo.of.x.z}
		\tilde{x} = \sum_{i=0}^{m} \mu_{i} v(i) 
			= t  \sum_{i=0}^{n} \beta_{v(i)}(x) v(i) + (1-t) \sum_{i=m+1}^{n} \mu_{i} v(i)  \\
			= \sum_{i=0}^{m} t \, \beta_{v(i)}(x) v(i) 
			+ \sum_{i=m+1}^{n} \bigl[ t \, \beta_{v(i)}(x) - (t-1) \mu_{i} \bigr] v(i).
	\end{multline}
Then by geometric independence of $v(0), \ldots, v(n)$ we have
	\begin{equation}  \label{E:mu.t.beta}
		\mu_{i} = t  \beta_{v(i)}(x), \quad i = 0, \ldots, m  \quad \text{ and } \quad
			\mu_{i} = \frac{t}{t-1} \beta_{v(i)}(x), \quad i = m+1, \ldots, n.
	\end{equation}
Let $b = \sum_{i=0}^{m} \beta_{v(i)}(x)$. Since $x \in \text{Int} \, \rho$, we have 
$b \in (0,1)$. From \eqref{E:mu.t.beta} and the fact that $\sum_{i=0}^{m} \mu_{i} = 1$ we see $t = 1/b > 1$. In particular, $z$ and $\tilde{x}$ are unique if they exist. If it exists, denote that $z$ by $z_{0}$.

Next, we prove existence of $z_{0}$. Let $t = 1/b$. Then it is easy to see that if $\mu_{0}, \ldots, \mu_{n}$ are defined by \eqref{E:mu.t.beta} then 
	\[
	\sum_{i=0}^{m} \mu_{i} = 1 = \sum_{i=m+1}^{n} \mu_{i}.
	\]
Hence, $z_{0} := \sum_{i=m+1}^{n} \mu_{i} v(i) \in \xi$ and 
$\tilde{x} := \sum_{i=0}^{m} \mu_{i} v(i) \in \rho \cap \tau$ and \eqref{E:xtilde.is.convx.combo.of.x.z} holds. Since $x \in  \text{Int} \, \rho$, we have $\beta_{v(i)}(x) > 0$ for $i \in \NN_{n}$. Therefore, $\mu_{i} > 0$ for $i = 1, \ldots, m$. Thus, $\tilde{x} \in \text{Int} \, (\rho \cap \tau)$. I.e., $z_{0} \in \xi$, $x$, and $\tilde{x} \in \text{Int} \, (\rho \cap \tau)$ lie on the same line. This proves the claim. Define $\tilde{y} \in \rho \cap \tau$ similarly. It has similar  properties.

The idea behind the rest of the proof is to first show that 
	\begin{equation}  \label{E:x.y.tilde.bckwrds.tringle.ineq}
		| x-\tilde{x} |+ |\tilde{x} - \tilde{y}| + | \tilde{y} - y | \leq K' \, |x - y|, 
	\end{equation}
where $K' = K'(\rho, \tau) < \infty$ depends only on $\rho$ and $\tau$, not on $x$ or $y$. Notice that $x$ and $\tilde{x}$ lie in the same simplex in $P$, \emph{viz.} $\rho$. Similarly, $\tilde{x}$ and $\tilde{y}$ both lie in $\rho \cap \tau \in P$. The points $\tilde{y}$ and $y$ also lie in the same simplex in $P$. So we may apply \eqref{E:beta.unif.Lip.on.each.simplex} to each term 
in $| x-\tilde{x} |+ |\tilde{x} - \tilde{y}| + | \tilde{y} - y |$ and then maximize $K'(\rho, \tau)$ over appropriate pairs $\rho, \tau \in P$.

The simplex $\rho \cap \tau$ lies on a unique plane, $\Pi_{\rho \cap \tau}$,  of minimum dimension. (See \eqref{E:formla.for.affine.plane}.)  ($\Pi_{\rho \cap \tau}$ might not pass through the origin.) So, e.g., if $\rho \cap \tau$ is a single point $v$ (i.e., $\rho \cap \tau$ 0-dimensional) 
then $\Pi_{\rho \cap \tau} = \{ v \}$. Now, $x \in \text{Int} \, \rho$ so $x \notin \Pi_{\rho \cap \tau}$.
Let $\hat{x} \in \Pi_{\rho \cap \tau}$ be the orthogonal projection of $x$ onto $\Pi_{\rho \cap \tau}$, i.e., 
$\hat{x}$ is the closest point of $\Pi_{\rho \cap \tau}$ to $x$. Note that $\hat{x}$ may not lie 
in $\rho \cap \tau$. Define $\hat{y}$ similarly. Let $x_{0}$ be an arbitrary point 
in $\text{Int} \, (\rho \cap \tau)$. E.g., $x_{0}$ might be the barycenter of $\rho \cap \tau$. (See \eqref{E:barycenter.of.sigma}.)  In any case, $x_{0}$ need only depend on $\rho \cap \tau$, not on $x$ or $y$.Let 
	\begin{equation}  \label{E:y0.=.x0}
		y_{0} := x_{0}.
	\end{equation}   
Then by \eqref{E:Int.rho.cuts.sigma.then.rho.in.sigma}, there exists $r > 0$ s.t.\ the distance 
from $x_{0} = y_{0}$ to any face of $\rho$ or $\tau$ that does not itself have $\rho \cap \tau$ as a face is at least $2r$. We may assume $r$ only depends on $\rho \cap \tau$, not on $x$ or $y$.

\emph{Claim:} 
	\begin{equation}  \label{E:xdot.in.rho.ydot.in.tau}
		\dot{x} := x_{0} + |x - \hat{x}|^{-1} r (x - \hat{x}) \in \rho \text{ and }
			\dot{y} := y_{0} + |y - \hat{y}|^{-1} r (y - \hat{y}) \in \tau. 
	\end{equation}
First, note that  
	\begin{equation}  \label{E:x0.+.t.x.-.xhat.in.rho}
		\text{for $t > 0$ sufficiently small, } x_{0} + t (x - \hat{x}) \in \text{Int} \, \rho.
	\end{equation}
To see this, observe that by \eqref{E:formla.for.affine.plane} we can write
	\[
		\hat{x} = \sum_{i=0}^{m} \zeta_{i} v(i),
	\]
where $v(0), \ldots, v(m)$ are the vertices of $\rho \cap \tau$; $\zeta_{0}, \ldots, \zeta_{m} \in \RR$; 
and $\zeta_{0} + \cdots + \zeta_{m} = 1$. (But the $\zeta_{i}$'s do not have to be nonnegative.)  Moreover, since $x_{0}$ is an interior point of $\rho \cap \tau$ we have
	\[
		\beta_{v(i)}(x_{0}) > 0, \text{ for } i = 0, \ldots, m, 
			\text{ but } \beta_{v(i)}(x_{0}) = 0 \text{ for } i = m+1, \ldots, n.
	\]
Let $t > 0$. Then
	\begin{equation}  \label{E:x0.+.t.x.-.xhat.in.terms.of.vs}
		x_{0} + t (x - \hat{x}) 
			= \sum_{i=0}^{m} \bigl( \beta_{v(i)}(x_{0}) - t \zeta_{i} + t \beta_{v(i)}(x) \bigr) v(i) 
			+ t \sum_{i=m+1}^{n} \beta_{v(i)}(x) v(i).
	\end{equation}
Since $\beta_{v(i)}(x_{0}) > 0$ for $i = 0, \ldots, m$, for $t > 0$ sufficiently small 
$\beta_{v(i)}(x_{0}) - t \zeta_{i} > 0$ for $i = 0, \ldots, m$. So certainly 
$\beta_{v(i)}(x_{0}) - t \zeta_{i} + t \beta_{v(i)}(x) > 0$ for $i = 0, \ldots, m$. I.e., the coefficients in \eqref{E:x0.+.t.x.-.xhat.in.terms.of.vs} are all strictly positive. Finally, the sum of the coefficients satisfies
	\begin{align*}
		\sum_{i=0}^{m} \bigl( \beta_{v(i)}(x_{0}) 
		  - t \zeta_{i} + t \beta_{v(i)}(x) \bigr) 
				+ t \sum_{i=m+1}^{n} \beta_{v(i)}(x)
			&= \sum_{i=0}^{m} \beta_{v(i)}(x_{0}) - t \sum_{i=0}^{m} \zeta_{i} 
				+ t \sum_{i=0}^{n} \beta_{v(i)}(x) \\
			&= 1 - t + t \\
			&= 1.
	\end{align*}
That completes the proof of \eqref{E:x0.+.t.x.-.xhat.in.rho}.

Now suppose $\dot{x}$ defined by \eqref{E:xdot.in.rho.ydot.in.tau} does \emph{not} lie 
in $\rho$. Let $\Pi_{\rho}$ be the smallest plane in $\RR^{N}$ containing $\rho$. 
So $\Pi_{\rho \cap \tau} \subset \Pi_{\rho}$. By \eqref{E:formla.for.affine.plane}, we have
	\[
		\Pi_{\rho} = \left\{ \sum_{i=0}^{n} \gamma_{i} v(i) : \sum_{i=0}^{n} \gamma_{i} = 1 \right\}
		         = \left\{ v(0) + \sum_{i=1}^{n} \gamma_{i} \bigl( v(i) - v(0) \bigr) : 
		            \gamma_{1}, \ldots, \gamma_{n} \in \RR \right\},
	\]
where $v(0), \ldots, v(n)$ are the vertices of $\rho$. Since $v(1) - v(0), \ldots, v(n) - v(0)$ are linearly independent, the map that takes a point 
$\sum_{i=0}^{n} \gamma_{i} v(i) \in \Pi_{\rho}$ to the vector $\gamma_{0}, \ldots, \gamma_{n}$ is well-defined and continuous. Now $x_{0} \in \rho \cap \tau \subset \Pi_{\rho}$, $x \in \rho \subset \Pi_{\rho}$, and $\hat{x} \in \Pi_{\rho \cap \tau} \subset \Pi_{\rho}$. Moreover, the coefficients of $x_{0}$, $x$, and $\hat{x}$ in the expression for $\dot{x}$ in \eqref{E:xdot.in.rho.ydot.in.tau}, \emph{viz.}, 1, $r/|x - \hat{x}|$, and $-r/|x-\hat{x}|$ sum to 1. It follows that  $\dot{x} \in \Pi_{\rho}$. 
Hence, we can write $\dot{x} = \sum_{i=0}^{n} \zeta_{i} v(i)$ with $\zeta_{0} + \cdots + \zeta_{n} = 1$. 

Let $S$ be the line segment joining $x_{0}$ and $\dot{x}$. I.e., 
	\begin{equation}  \label{E:line.segment.S.defn}
		S = \bigl\{ x_{0} + t(x - \hat{x}) : 0 \leq t \leq r/|x-\hat{x}| \bigr\}.
	\end{equation}
By \eqref{E:x0.+.t.x.-.xhat.in.rho} for some $t \in (0, r/|x-\hat{x}|)$ we have
	\begin{equation}   \label{E:x.prime.defn}
		x' := x_{0} + t(x-\hat{x}) \in (\text{Int} \, \rho) \cap S.
	\end{equation}
Since $x' \in \text{Int} \, \rho$, the coefficients in the representation of $x'$ as a linear combination of 
\linebreak $v(0), \ldots, v(n)$ must all be strictly positive. 
Since by assumption $\dot{x} \notin \rho$, one or more of the coefficients, $\zeta_{0}, \ldots, \zeta_{n}$, of $v(0), \ldots, v(n)$ for $\dot{x}$ must be strictly negative. Therefore, somewhere between $x'$ and $\dot{x}$ the segment $S$ must cross the boundary 
$\text{Bd} \, \rho$. 
Let $w \in \text{Bd} \, \rho$ be the point of intersection. Thus, for some $s \in (t, r/|x-\hat{x}|)$ we have 
	\begin{equation}  \label{E:w.in.terms.of.x0.x.xhat}
		w = x_{0} + s (x - \hat{x}).
	\end{equation}
	
Let $\omega$ be the, necessarily proper, face of $\rho$ s.t.\ $w \in \text{Int} \, \omega$. (See \eqref{E:x.in.exctly.1.simplex.intrr}.) Now, $\rho \cap \tau$ cannot be a face 
of $\omega$. For suppose  $\rho \cap \tau \subset \omega$. Note that $w \ne x_{0}$, because otherwise $s (x - \hat{x}) = 0$ in \eqref{E:w.in.terms.of.x0.x.xhat}, an impossibility since $x \ne \hat{x}$ and $s > 0$. Hence, under the assumption that $\rho \cap \tau \subset \omega$ the segment $S$ contains two distinct points of $\omega$, \emph{viz.}, $x_{0} \in \rho \cap \tau$ and $w$. As a proper face of $\rho$, the simplex 
$\omega$ is defined by the vanishing of some set of barycentric coordinates. Thus, there exists a nonempty proper subset 
$J$ of $\{0, \ldots, n  \}$ s.t.\
	\[
		\omega = \left\{ \sum_{j = 0}^{n} \beta_{j} v(j) : \beta_{j} \geq 0 \; (j = 0, \ldots, n), 
			\beta_{j} = 0 \text{ if } j \in J, \text{ and } \sum_{j=0}^{n} \beta_{j} = 1 \right\}.
	\]
Since $x, \hat{x} \in \Pi_{\rho}$, for some $\gamma_{0}, \ldots, \gamma_{n} \in \RR$ we have
	\[
		x - \hat{x} =  \sum_{j=0}^{n} \gamma_{j} v(j), 
			\text{ where } \sum_{j=0}^{n} \gamma_{j} = 0.
	\]
Under the hypothesis that $\rho \cap \tau \subset \omega$, we have $w, x_{0} \in \omega$. In particular, we have $\beta_{v(j)}(x_{0}) = 0$ for $j \in J$. It follows from \eqref{E:w.in.terms.of.x0.x.xhat} that $\gamma_{j} = 0$ if $j \in J$. Hence, by \eqref{E:line.segment.S.defn} for every $x'' \in S \subset \Pi_{\rho}$ we can write (uniquely)
	\[
	    x'' = \sum_{j \in J^{c}} \alpha_{j} v(j), \text{ where } 
		  \sum_{j \in J^{c}} \alpha_{j} = 1.
	\]
(Here, $J^{c} = \{ j = 0, \ldots, n : j \notin J \}$.)  In particular, 
$S \cap (\text{Int} \, \rho) = \varnothing$. 
But by \eqref{E:x.prime.defn}, $x' \in S \cap (\text{Int} \, \rho)$. Contradiction. This proves $\rho \cap \tau$ cannot be a face of $\omega$. 

Since $\rho \cap \tau$ is not a face of $\omega$, by choice of $r > 0$ the distance 
from $x_{0}$ to $\omega$ is at least $2r$. Since $\omega$ lies between $x_{0}$ and $\dot{x}$ 
along $S$ we have by \eqref{E:xdot.in.rho.ydot.in.tau} 
	\[
		r = | \dot{x} + x_{0} | \geq 2r > 0.
	\]
This contradiction proves the claim \eqref{E:xdot.in.rho.ydot.in.tau}.

\emph{Claim:  The angle between $x-\hat{x}$ and $y-\hat{y}$ is bounded away from 0.}  I.e., there exists $\gamma \in (0,1)$ independent of $x \in \text{Int} \, \rho$ and $y \in \text{Int} \, \tau$ (i.e., $\gamma$ only depends on $\rho$ and $\tau$) s.t.\
	\begin{equation}  \label{E:angle.tween.x-xhat.y-yhat.ain't.0}
		(x-\hat{x}) \cdot (y-\hat{y}) \leq \gamma |x-\hat{x}||y-\hat{y}|,
	\end{equation}
where, as usual, ``$\cdot$'' indicates the usual Euclidean inner product.
Suppose \eqref{E:angle.tween.x-xhat.y-yhat.ain't.0} is false. Then there exist sequences 
$\{ x_{n} \} \subset \text{Int} \, \rho$, $\{ y_{n} \} \subset \text{Int} \, \tau$ s.t.\ 
	\[
		\frac{(x_{n}-\hat{x}_{n}) \cdot (y_{n}-\hat{y}_{n})}{|x_{n}-\hat{x}_{n}||y_{n}-\hat{y}_{n}|} 
		           \to 1,
	\]
where $\hat{x}_{n}$ ($\hat{y}_{n}$) is the orthogonal projection of $x_{n}$ (resp. $y_{n}$) onto $\Pi_{\rho \cap \tau}$. Define $\dot{x}_{n}$ as in \eqref{E:xdot.in.rho.ydot.in.tau} with $x$ and $\hat{x}$ replaced by $x_{n}$ and $\hat{x}_{n}$, resp. Define $\dot{y}_{n}$ similarly.  By definition of $\dot{x}_{n}$ and $\hat{x}_{n}$ the vector $\dot{x}_{n} - x_{0}$ has length $r > 0$ and is orthogonal 
to $\Pi_{\rho \cap \tau}$. Ditto for $\dot{y}_{n} - y_{0}$. But $x_{0} \in \rho \cap \tau  \subset \Pi_{\rho \cap \tau}$. Hence, $dist(\dot{x}_{n}, \rho \cap \tau) \geq r$. Moreover, by \eqref{E:xdot.in.rho.ydot.in.tau}, $\dot{x}_{n} \in \rho$. 
Similarly, $dist(\dot{y}_{n}, \rho \cap \tau) \geq r$ and $\dot{y}_{n} \in \tau$.
	 
Therefore, by compactness of $\rho$ and $\tau$, we may assume $\dot{x}_{n} \to \dot{x}_{\infty} \in \rho$ and $\dot{y}_{n} \to \dot{y}_{\infty} \in \tau$. We must have $| \dot{x}_{\infty} - x_{0} | = r$, $| \dot{y}_{\infty} - y_{0} | = r$, $dist(\dot{x}_{\infty}, \rho \cap \tau) \geq r$, and $dist(\dot{y}_{\infty}, \rho \cap \tau) \geq r$. 
In particular, 
	\begin{equation}  \label{E:x.dot.infty.in.rho.less.rho.cap.tau}
		\dot{x}_{\infty} \in \rho \setminus (\rho \cap \tau) \text{ and }
			\dot{y}_{\infty} \in \tau \setminus (\rho \cap \tau)
	\end{equation}

Now, by definition of $\{ x_{n} \}$, $\{ y_{n} \}$, $\{ \dot{x}_{n} \}$, and $\{ \dot{y}_{n} \}$, we have
	\begin{equation*}
		(\dot{x}_{n} - x_{0}) \cdot (\dot{y}_{n} - y_{0})  
			= r^{2} \frac{(x_{n}-\hat{x}_{n}) \cdot (y_{n}-\hat{y}_{n})}
					{|x_{n}-\hat{x}_{n}||y_{n}-\hat{y}_{n}|}
			         \to r^{2} = | \dot{x}_{\infty} - x_{0} | | \dot{y}_{\infty} - y_{0} | \text{ as } n \to \infty.
	\end{equation*} 
But, 
	\[
		 (\dot{x}_{n} - x_{0}) \cdot (\dot{y}_{n} - y_{0})
			  \to (\dot{x}_{\infty} - x_{0}) \cdot (\dot{y}_{\infty} - y_{0}) \text{ as } n \to \infty.
	\]
This means $\dot{x}_{\infty} - x_{0}$ and $\dot{y}_{\infty} - y_{0}$ are positive multiples of each other. But $\dot{x}_{\infty} - x_{0}$ and $\dot{y}_{\infty} - y_{0}$ have the same length $r$. 
Hence, $\dot{x}_{\infty} - x_{0} = \dot{y}_{\infty} - y_{0}$. However, by \eqref{E:y0.=.x0}, $y_{0} = x_{0}$. Therefore, $\dot{x}_{\infty} = \dot{y}_{\infty}$. In particular, 
$\dot{x}_{\infty}, \dot{y}_{\infty} \in \rho \cap \tau$. This contradicts \eqref{E:x.dot.infty.in.rho.less.rho.cap.tau}. The claim \eqref{E:angle.tween.x-xhat.y-yhat.ain't.0} follows.

By definition of $\hat{x}$ and $\hat{y}$ and \eqref{E:angle.tween.x-xhat.y-yhat.ain't.0}, we have 
	\begin{align}  \label{E:expand.length.of.x-y.using.hats}
		|x - y|^{2} &= \bigl| (x-\hat{x}) + (\hat{x} - \hat{y}) + (\hat{y} - y) \bigr|^{2} \notag \\
		  &= | x-\hat{x} |^{2} + | \hat{x} - \hat{y} |^{2} - 2 (x-\hat{x})  \cdot (y-\hat{y}) 
		    + | \hat{y} - y |^{2}            
		             \notag \\
		  &\geq | x-\hat{x} |^{2} + | \hat{x} - \hat{y} |^{2} - 2 \gamma |x-\hat{x}||y-\hat{y}| 
		        + | y  - \hat{y} |^{2} \\
		  &= \lambda \bigl( | x-\hat{x} |^{2} + | y  - \hat{y} |^{2} \bigr) + | \hat{x} - \hat{y} |^{2} + 
		         \gamma \bigl( | x-\hat{x} | - | y  - \hat{y} | \bigr)^{2} \notag \\
		  &\geq \lambda \bigl( | x-\hat{x} |^{2} + | y  - \hat{y} |^{2}
		    + | \hat{x} - \hat{y} |^{2}  \bigr). 
		             \notag
	\end{align}
Applying \eqref{E:n.c.sqrd.sum.ineq} twice to \eqref{E:expand.length.of.x-y.using.hats} we get
	\begin{align*}
		|x - y|^{2} &\geq \frac{1-\gamma}{4} 
	                      \Bigl(2  \bigl[ | x-\hat{x} | + | y  - \hat{y} | \bigr]^{2}
	                        + 4| \hat{x} - \hat{y} |^{2}  \Bigr) \\
	                 &\geq \frac{1-\gamma}{4} 
	                      \Bigl(2  \bigl[ | x-\hat{x} | + | y  - \hat{y} | \bigr]^{2}
	                        + 2| \hat{x} - \hat{y} |^{2}  \Bigr) \\
	                  &\geq \frac{1-\gamma}{4} 
	                      \bigl[ | x-\hat{x} | + | y  - \hat{y} | + | \hat{x} - \hat{y} | \bigr]^{2}. \\
	\end{align*}
We conclude
	\begin{equation}  \label{E:|x-y|.dominates.mult.of.sum.of.tilde.pieces}
		\frac{2}{\sqrt{1-\gamma}} |x - y| 
		\geq  | x-\hat{x} |+ | \hat{x} - \hat{y} | + | y  - \hat{y} |  ,
			\text{ for } x \in \text{Int} \, \rho, \; y \in \text{Int} \, \tau.
	\end{equation}

\emph{Claim: The angle, $\theta$, between $x - \tilde{x}$ and $\Pi_{\rho \cap \tau}$ is bounded away from 0.}   Since $\hat{x}$ is the orthogonal projection of $x$ onto $\Pi_{\rho \cap \tau}$, we have 
that $\theta$ is the angle between $x - \tilde{x}$ and $\hat{x} - \tilde{x}$ 
and $\sin \theta =  |x - \hat{x}|/|x - \tilde{x}|$. By definition of $\hat{x}$, $|x - \tilde{x}|/|x - \hat{x}| \ge 1$. Therefore, $\theta$ being bounded away from 0 is equivalent to
	\begin{multline}  \label{E:x.minus.xtilde.over.x.minus.xhat.bdd}
		1/\sin \theta = |x - \tilde{x}|/|x - \hat{x}| \text{ is bounded above by some }  \\
		    \alpha \in (1, \infty)  \text{ independent of } x \in \text{Int} \, \rho.
	\end{multline}
And similarly for $y$, $\tilde{y}$, and $\hat{y}$.

If $z \in \xi$ (the face of $\rho$ opposite $\rho \cap \tau$), let $\hat{z}$ denote the orthogonal projection of $z$ onto $\Pi_{\rho \cap \tau}$. Recall that $\tilde{x}$, $x$, and $z_{0}$ lie on the same line. Taking orthogonal projections, we see that $\tilde{x}$, $\hat{x}$, and $\hat{z}_{0}$ lie on the same line in $\Pi_{\rho \cap \tau}$. Therefore, by similarity of triangles\footnote{To see all this analytically, 
let $c = |x-\tilde{x}|/|z_{0} - \tilde{x}|$. ($|z_{0} - \tilde{x}| > 0$, since $\tilde{x} \in \rho \cap \tau$ 
and $z_{0} \in \xi$, the face opposite $\rho \cap \tau$.)  Then 
	\begin{equation} \label{E:x.in.terms.of.c.z0.xtilde}
		x = c(z_{0} - \tilde{x}) + \tilde{x},
	\end{equation}
since $x$ lies on the line segment joining $z_{0}$ and $\tilde{x}$. 
Let $\ddot{x} = c(\hat{z}_{0} - \tilde{x}) + \tilde{x}$. Then $\ddot{x}$ lies on the line joining $\tilde{x}$ and $\hat{z}_{0}$. (In particular, $\ddot{x} \in \Pi_{\rho \cap \tau}$.)  But it is easy to see 
from \eqref{E:x.in.terms.of.c.z0.xtilde} that $x - \ddot{x} = c(z_{0} - \hat{z}_{0}) \perp \Pi_{\rho \cap \tau}$. I.e., 
	\begin{equation}  \label{E:xhat.=.xdoubledot}
		\hat{x} = \ddot{x} = c(\hat{z}_{0} - \tilde{x}) + \tilde{x}.
	\end{equation}
Thus, $z_{0} - \tilde{x}$, $x - \tilde{x}$, $\hat{x} - \tilde{x}$, and $\hat{z}_{0} - \tilde{x}$ lie in the subspace spanned by $z_{0} - \tilde{x}$ and $\hat{z}_{0} - \tilde{x}$ and
		\[
			\frac{|x - \tilde{x}|}{|x - \hat{x}|} 
			  = \frac{\Bigl| \bigl[ c(z_{0} - \tilde{x}) + \tilde{x} \bigr] - \tilde{x} \Bigr|}
			            {\Bigl| \bigl[ c(z_{0} - \tilde{x}) + \tilde{x} \bigr] 
			                     - \bigl[ c(\hat{z}_{0} - \tilde{x}) + \tilde{x} \bigr]  \Bigr|} 
			  = \frac{|z_{0} - \tilde{x}|}{|z_{0} - \hat{z}_{0}|}
		\]
by \eqref{E:x.in.terms.of.c.z0.xtilde} and \eqref{E:xhat.=.xdoubledot}.},  
		\[
			\frac{|x - \tilde{x}|}{|x - \hat{x}|} 
			  = \frac{|z_{0} - \tilde{x}|}{|z_{0} - \hat{z}_{0}|}.
		\]
But since $\xi$ and 
$\rho \cap \tau$ are disjoint and compact, $|z - \hat{z}|$ is bounded below and $|z - w|$ is bounded above in $(z,w) \in \xi \times (\rho \cap \tau)$. The claim
\eqref{E:x.minus.xtilde.over.x.minus.xhat.bdd} follows. Of course, the same thing goes for $y$ and we may assume the same $\alpha$ works for both $\rho$ and $\tau$. 

It follows from \eqref{E:x.minus.xtilde.over.x.minus.xhat.bdd} and the Pythagorean theorem that 
	\begin{equation}  \label{E:xtilde.minus.xhat.bound}
		|\tilde{x} - \hat{x}| \leq \sqrt{\alpha^{2} -1} \;   |x - \hat{x}| < \alpha |x - \hat{x}|. 
		   \text{ Similarly for } y, \tilde{y}, \text{ and } \hat{y}.
	\end{equation}
Consequently,  
	\begin{multline*}
		|\tilde{x} - \tilde{y}| \leq |\tilde{x} - \hat{x}| + |\hat{x} - \hat{y}| + |\hat{y} - \tilde{y}|
			\leq \alpha |x - \hat{x}| + |\hat{x} - \hat{y}| + \alpha |y - \hat{y}| \\
			\leq \alpha |x - \hat{x}| + 2 \alpha |\hat{x} - \hat{y}| + \alpha |y - \hat{y}|,
	\end{multline*}
since $\alpha > 1$. Hence,
	\[
		|\hat{x} - \hat{y}| \geq \frac{1}{2 \alpha} |\tilde{x} - \tilde{y}| - \frac{1}{2} |x - \hat{x}| 
			- \frac{1}{2} |y - \hat{y}|.
	\]
Substituting this into \eqref{E:|x-y|.dominates.mult.of.sum.of.tilde.pieces} we get
	\begin{equation*}
		\frac{2}{\sqrt{1-\gamma}} |x - y| 
			\geq  \frac{1}{2} | x-\hat{x} |+ \frac{1}{2 \alpha} |\tilde{x} - \tilde{y}| 
				+ \frac{1}{2}| y  - \hat{y} |.
	\end{equation*}
Therefore, by \eqref{E:x.minus.xtilde.over.x.minus.xhat.bdd} again,
	\begin{equation}
		\frac{2}{\sqrt{1-\gamma}} |x - y| 
		  \geq  \frac{1}{2 \alpha}  \bigl( | x-\tilde{x} |
			  + |\tilde{x} - \tilde{y}| + | y - \tilde{y} | \bigr).
	\end{equation}
I.e., if \eqref{E:rho.intersect.but.aren't.subsets} holds
	\begin{equation}  \label{E:mult.of.x.minus.y.dominates.x.minus.xtilde.etc}
		\frac{4 \alpha}{\sqrt{1-\gamma}} |x - y| 
			\geq  | x-\tilde{x} |+ |\tilde{x} - \tilde{y}| + | y - \tilde{y} |.
	\end{equation}

Let $K' = K'(\rho, \tau) := \tfrac{4 \alpha}{\sqrt{1-\gamma}}$. 
Then \eqref{E:x.y.tilde.bckwrds.tringle.ineq} holds. (Maximizing over all appropriate $\rho, \tau \in P$ yields corollary \ref{C:reverse.triangle.ineq.in.simp.cmplxs}.) \eqref{E:mult.of.x.minus.y.dominates.x.minus.xtilde.etc} 
and \eqref{E:beta.unif.Lip.on.each.simplex} together imply 
	\begin{align*}
		\bigl| \blds{\beta}(x) - \blds{\beta}(y)  \bigr| 
			& \leq \bigl| \blds{\beta}(x) - \blds{\beta}(\tilde{x})  \bigr|
				+ \bigl| \blds{\beta}(\tilde{x}) - \blds{\beta}(\tilde{y})  \bigr|
				+ \bigl| \blds{\beta}(\tilde{y}) - \blds{\beta}(y)  \bigr| \\
			&\leq K \bigl( | x-\tilde{x} |+ | \tilde{x} - \tilde{y} | + | y - \tilde{y} |  \bigr) \\
			&\leq K K'(\rho, \tau) |x - y|.
	\end{align*}
Now maximize over all $\rho, \tau \in P$. This completes the proof.
\end{proof}

Let $\sigma \in P$ and let
   \[
      \overline{\text{St}} \, \sigma 
         = \bigcup_{\sigma \subset \omega \in P} \omega.
   \]
$\overline{\text{St}} \, \sigma$ is the ``closed star'' of $\sigma$ (Munkres \cite[p.\ 371]{jrM84}). By \eqref{E:intersection.of.simps} $\overline{\text{St}} \, \sigma$ is the union of all simplices in $P$ having 
$\sigma$ as a face. In particular, $\sigma \subset \overline{\text{St}} \, \sigma$. 
Let $\text{Lk} \, \sigma$ be the union of all simplices lying in 
$\overline{\text{St}} \, \sigma$ that do \emph{not} intersect 
$\sigma$. $\text{Lk} \, \sigma$ is the ``link'' of $\sigma$. 
We may have $\overline{\text{St}} \, \sigma = \sigma$, which implies 
$\text{Lk} \, \sigma = \varnothing$. This can happen, e.g., if $\dim \sigma = \dim P$. 
If $\rho \in P$, 
$\rho \subset  \overline{\text{St}} \, \sigma$, $\sigma \neq \rho$, and $\omega$ is the face of $\rho$ opposite $\sigma$, 
then $\omega \subset \text{Lk} \, \sigma$. 
Thus, $\overline{\text{St}} \, \sigma = \sigma$ if and only if $\text{Lk} \, \sigma = \varnothing$.

The ``star'', $\text{St} \, \sigma$ of $\sigma$ is the union of the interiors of all simplices of $P$ having $\sigma$ as a face (Munkres \cite[p.\ 371]{jrM84}). (If $ \overline{\text{St}} \, \sigma = \sigma$, then $\text{St} \, \sigma = \text{Int} \, \sigma$.) We have
	\begin{multline}  \label{E:St.sigma.is.open}
	        \text{St} \, \sigma 
	             = \bigl\{ y \in |P| : \beta_{v}(y) > 0 \text{ for every } v \in \sigma^{(0)} \bigr\}
	               \text{ so  } \text{St} \, \sigma \text{ is open in } |P|. \\
	                \text{Moreover, } \text{Int} \, \sigma \subset \text{St} \, \sigma, 
	                       (\text{St} \, \sigma) \cap (\text{Lk} \, \sigma) = \varnothing, \text{ and } 
	                          (\text{St} \, \sigma) \cap (\text{Bd} \, \sigma) = \varnothing.
	\end{multline}
(\emph{Proof:}  $\rho \in P$ has $\sigma$ as a face if and only if $\sigma^{(0)} \subset \rho^{(0)}$. But, by \eqref{E:criterion.for.int.simp}, $x = \sum_{v \in \rho^{(0)}} \beta_{v}(x) \in \text{Int} \, \rho$ if and only if $\beta_{v}(x) > 0$ for every $v \in \rho^{(0)}$. Hence, if $x \in \text{Int} \, \rho$ and $\rho$ has $\sigma$ as a face then $\beta_{v}(x) > 0$ for every $v \in \sigma^{(0)}$. Conversely, suppose $x \in |P|$ and 
$\beta_{v}(x) > 0$ for every $v \in \sigma^{(0)}$. Then obviously, if $\rho$ is the simplex in $P$ with 
$x \in \text{Int} \, \rho$, we have $\sigma^{(0)} \subset \rho^{(0)}$ so $\rho \in P$ has $\sigma$ as a face. Thus, $x \in \text{Int} \, \rho \subset \text{St} \, \sigma$. 
In particular, $\text{Int} \, \sigma \subset \text{St} \, \sigma$. Since $\beta_{v}$ ($v \in \sigma^{(0)}$) are continuous, it follows that $\text{St} \, \sigma$ is open. Moreover, if $x \in (\text{Lk} \, \sigma) \cup (\text{Bd} \, \sigma)$ then $\beta_{v}(x) = 0$ 
for some $v \in \sigma^{(0)}$. Hence, neither $\text{Lk} \, \sigma$ nor $\text{Bd} \, \sigma$ intersects $\text{St} \, \sigma$.)

Let $\sigma \in P$. Observe that, true to their names, both $\overline{\text{St}} \, \sigma$ and 
$\text{St} \, \sigma$ are ``starlike'' w.r.t.\ any $x \in \text{Int} \, \sigma$. 
I.e., if $y \in \overline{\text{St}} \, \sigma$ then the line segment joining $x$ and $y$ lies entirely 
in $\overline{\text{St}} \, \sigma$. The same goes for $y \in \text{St} \, \sigma$. \emph{Claim:} $|P|$ is locally arcwise connected (Massey \cite[p.\ 56]{wsM67.Massey}). To see this, let $x \in |P|$ and let $\sigma$ be the unique simplex in $P$ s.t.\ $x \in \text{Int} \, \sigma$. (See \eqref{E:x.in.exctly.1.simplex.intrr}.)  $\text{St} \, \sigma$ is an open neighborhood of $x$. Let $r > 0$ be so small that the open ball $B_{r}(x)$, of radius $r$ centered at $x$ satisfies $B_{r}(x) \cap |P| \subset \text{St} \, \sigma$. If $y, z \in B_{r}(x)  \cap |P|$, then the line segments joining $y$ to $x$ and $x$ to $z$ also lie in $B_{r}(x)  \cap |P|$. I.e., $B_{r}(x)  \cap |P|$ is path connected. This proves the claim. Thus, if $|P|$ is connected it is also arcwise connected.

	  \begin{lemma}  \label{L:local.finiteness.and.compactness}
 Let $P$ be a simplicial complex lying in a finite dimensional Euclidean space, $\RR^{N}$. Suppose every $x \in |P|$ has a neighborhood, open in $\RR^{N}$, intersecting only finitely many simplices in $P$. Then the following hold.
		\begin{enumerate}
		\renewcommand{\theenumi}{\roman{enumi}}
		\item $P$ is ``locally finite'':  
		      Each $v \in P^{(0)}$ belongs to only finitely many simplices 
		     in $P$. \label{I:P.locally.finite}
		\item $|P|$ is locally compact.  \label{I:local.compactness}
		\item $|P|$ is a subspace of $\RR^{N}$. I.e., the polytope topology of $|P|$ coincides with the topology that $|P|$ inherits from $\RR^{N}$.  
		         \label{I:|P|.is.subspace}
		\end{enumerate}
	  \end{lemma}
  \begin{proof} 
Suppose $|P| \subset \RR^{N}$ and every $x \in |P|$ has a neighborhood open in $\RR^{N}$ and intersecting only finitely many simplices in $P$. Let $v \in P^{(0)}$. Then $v$ has a neighborhood $U$ that intersects only finitely many simplices in $P$. If $\sigma \in P$ and $v \in \sigma^{(0)}$, then $v \in \sigma \cap U$. I.e., $U$ intersects 
$\sigma$. Therefore, $v$ is a vertex of only finitely many simplices in $P$. This proves (\ref{I:P.locally.finite}). 
(See Munkres \cite[p.\ 11]{jrM84}.)

By item (\ref{I:P.locally.finite}) and Munkres \cite[Lemma 2.6, p.\ 11]{jrM84} we have that $|P|$ is locally compact. And by Munkres \cite[Exercise 9, p.\ 14]{jrM84}, the space $|P|$ is a subspace of $\RR^{N}$.
  \end{proof}  

  \begin{definition}  \label{D:subdivision}
A simplicial complex $P'$ in $\RR^{N}$ is a ``subdivision'' of $P$ (
Munkres \cite[p.\ 83]{jrM84}) if:
	\begin{enumerate}
		\item Each simplex in $P'$ is contained in a simplex of $P$.
		\item Each simplex in $P$ equals the union of finitely many simplices in $P'$.
	\end{enumerate}
  \end{definition}
In particular, a subdivision of a finite complex is finite. Suppose $P'$ is a subdivision of $P$. Then
	\begin{multline} \label{E:inclusion.of.subdiv.interiors}
		\text{If } \tau \in P' \text{ and } \sigma \in P \text{ is the smallest simplex (w.r.t.\ inclusion)  
		     in $P$ containing } \tau, \\
		             \text{ then } \text{Int} \, \tau \subset \text{Int} \, \sigma.
	\end{multline}
 For let $\tau \in P'$ have vertices $w_{0}, \ldots, w_{q} \in |P|$. Since $P'$ is a subdivision of $P$ there exists $\zeta \in P$ s.t.\ $\tau \subset \zeta$. Let $\zeta = \sigma$ be the smallest such simplex in $P$. Write $\sigma = \langle v_{0}, \ldots, v_{p} \rangle$. Then for some $\beta_{ij} \geq 0$ with $i = 0, \ldots, q$ and $j = 0, \ldots, p$ we have
	\begin{equation*}
		\sum_{j=0}^{p} \beta_{ij} = 1 \text{ and } w_{i} = \sum_{j=0}^{p} \beta_{ij} v_{j}, 
		        \quad i = 0, \ldots, q.
	\end{equation*}
Since $\sigma$ is minimal, for every $j = 0, \ldots, p$ there exists 
$i_{j} = 0, \ldots, q$ s.t.\ $\beta_{i_{j}j} > 0$. Let $x \in \text{Int} \, \tau$ then by \eqref{E:criterion.for.int.simp} there exist $\gamma_{0}, \ldots, \gamma_{q} >  0$ s.t.\
	\begin{equation*}
		x = \sum_{i=0}^{q} \gamma_{i} w_{j} 
		   = \sum_{j=0}^{p} \left( \sum_{i=0}^{q} \beta_{ij}  \gamma_{i} \right) v_{j}.
	\end{equation*}
But for $j = 0, \ldots, p$, we have
	\[
		\sum_{i=0}^{q} \beta_{ij}  \gamma_{i} \geq \beta_{i_{j}j}  \gamma_{i_{j}} > 0.
	\]
Hence, $x \in \text{Int} \, \sigma$ and \eqref{E:inclusion.of.subdiv.interiors} is proved.

The following notion is basic.
  \begin{definition}  \label{D:simplicial.map}
Let $P$ and $Q$ be simplicial complexes. A function $f : |P| \to |Q|$ is a ``simplicial map'' (Munkres \cite[p.\ 12]{jrM84}) from $P$ to $Q$ if whenever $v \in P^{(0)}$, then (1) $f(v) \in Q^{(0)}$, (2) if $v_{0}, \ldots, v_{p}$ span a simplex, $\sigma$, in $P$ then $f(v_{0}), \ldots, f(v_{p})$ are vertices of a simplex, $\tau$,  in $Q$, and (3)
if $x = \sum_{i=0}^{p} \beta_{i} v_{i} \in \sigma \;$ ($\beta_{0}, \ldots, \beta_{p} \geq 0$, $\sum_{i=0}^{p} \beta_{i} = 1$) then
	\begin{equation}   \label{E:simp.map.commutes.with.bary}
	     f \left( \sum_{i=0}^{p} \beta_{i} v_{i} \right) = \sum_{i=0}^{p} 
		  \beta_{i} \, f(v_{i}) \in \tau.
	\end{equation}
  \end{definition}

(If $v_{0}, \ldots, v_{p}$ span a simplex, $\sigma$, in $P$ then $f(v_{0}), \ldots, f(v_{p})$ are vertices of some $\tau$,  in $Q$. $f(v_{0}), \ldots, f(v_{p})$ may not span a simplex, because they might not be distinct and therefore not be geometrically independent.) 

Simplicial maps are continuous on the underlying spaces of their complexes. In fact, if $x \in |P|$ write $x$ in barycentric coordinates:
	\[
		x = \sum_{v \in P^{(0)}} \beta_{v}(x) v.
	\]
Clearly, 
	\[
		f(x) = \sum_{v \in P^{(0)}} \beta_{v}(x) f(v).
	\]
Thus, by proposition \ref{P:bary.coords.are.Lip}, we have 
	\begin{equation}   \label{E:simp.maps.are.Lip}
		\text{A simplicial map on a finite simplicial complex is Lipschitz.}
	\end{equation}
It is easy to see that the linearity of a simplicial map does not just apply to vertices:
	\begin{multline}  \label{E:simp.map.is.affine.on.simplx}
		\text{If } x_{1}, \ldots, x_{m} \in \sigma \in P \text{ and } \alpha_{1}, \ldots, \alpha_{m}  \\
			\text{ are non-negative and sum to 1, then} \\
				f \left( \sum_{i=1}^{m} \alpha_{i} x_{i} \right) 
					= \sum_{i=1}^{m} \alpha_{i} f(x_{i}).
	\end{multline}

We define a special kind of simplicial map.
  \begin{definition}  \label{D:simplicial.homeom}
If $g$ is a simplicial map of $|P|$ onto itself s.t.\ whenever $v_{0}, \ldots, v_{p} \in P^{(0)}$ then $v_{0}, \ldots, v_{p}$ span a simplex in $P$ if and only if $g(v_{0}), \ldots, g(v_{p})$ do, then we say that $g$ is a ``simplicial homeomorphism'' of $P$ onto itself (Munkres \cite[p.\ 13]{jrM84}).
  \end{definition}
Thus, if $g$ is a simplicial homeomorphism of $P$ onto itself we get to replace $|P|$ by $P$ in the description of $g$. Note that 
	\begin{multline} \label{E:simp.homeo.is.homeom}
		\text{If } g : P \to P\text{ is a simplicial homeomorphism then } \\
		         g : |P| \to|P| 
		       \text{ is a homeomorphism and,} \\
		       	\text{ if } \sigma \in P, \text{ then } 
		            g( \text{Int} \, \sigma) =  \text{Int} \, g(\sigma). 
	\end{multline}
(See \eqref{E:simp.map.commutes.with.bary} and \eqref{E:criterion.for.int.simp}.)

An important example of a simplicial homeomorphism is provided by the following.

  \begin{lemma}  \label{L:Cart.pwrs.of.simp.cmplx.can.be.perm.invar.cmplx}
Let $K$ be a finite simplicial complex and let $n = 2, 3, \ldots$. Suppose $|K| \subset \RR^{N}$, where $N = 0, 1, \ldots$. Let $S$ be the group of permutations of $1, \ldots, n$ and if $s \in S$, let $g_{s} :  \RR^{nN} \to \RR^{nN}$ apply $s$ to coordinates. I.e.,  
$g_{s}(x_{1}, x_{2}, \ldots, x_{n}) = (x_{i_{s(1)}}, x_{i_{s(2)}}, \ldots, x_{i_{s(n)}})$, for every $x_{1}, x_{2}, \ldots, x_{n} \in \RR^{N}$. Then there is a finite simplicial complex $P$ and triangulation $f : |P| \to |K|^{n}$ s.t.\ for every $s \in S$, we have that $f^{-1} \circ g_{s} \circ f$ is a simplicial homeomorphism from $P$ to itself. In fact, we may assume that $|P| = |K|^{n}$ and $f$ is the identity.
  \end{lemma}
  \begin{proof}
 Our proof is similar to that of Munkres \cite[Lemma 7.8, p.\ 75]{jrM66}. The result is trivial if $N = 0$ so suppose $N > 0$. We construct $P$. 
Suppose $|K|$ lies in $\RR^{N}$. So $|K|^{n} \subset \RR^{nN}$. Write points of $\RR^{nN}$ as $(x_{1}, \ldots, x_{n})$ 
with $x_{i}^{1 \times N} \in \RR^{N}$  ($i=1, \ldots n$). $|K|^{n}$ is the union of all products of the form 
    \begin{equation}  \label{E:c.=.simplex.prod}
      c := \sigma_{1} \times \cdots \times \sigma_{n} \subset \RR^{nN},
    \end{equation}
where $\sigma_{i} \in K$. Let $m_{i} := \dim \sigma_{i}$ ($i=1, \ldots n$). 

Define inequalities coordinate-wise. Let $i \in \NN_{n}$. By lemma \ref{L:simp.lin.ineqs}, if $0 < m_{i} < N$, we have 
	\begin{equation} \label{E:sigma.i.ineqs.eqs}
		\sigma_{i} = \bigl\{ x_{i}^{1 \times N} \in \RR^{N} : 
		  x_{i} A_{i} \, (I_{m_{i}}, \, - 1^{m_{i}}) \geq (y_{i}, -y_{i} 1^{m_{i}} - 1) 
		    \text{ and } x_{i} B_{i} = z_{i} \bigr\}, 
	\end{equation}
where $A_{i}^{N \times m_{i}}$ and $B_{i}^{N \times (N-m_{i})}$ are matrices of rank $m_{i}$ and $N-m_{i}$ resp.\ s.t.\ 
$B_{i}^{T} A_{i} = 0$, 
and $y_{i} \in \RR^{m_{i}}$ and $z_{i} \in \RR^{N-m_{i}}$ are row vectors. If $m_{i} = 0$ omit the inequality involving $A_{i}$. If $m_{i} = N$, omit the inequality involving $B_{i}$. But in that case we will still say ``$rank \, B_{i} = 0$''. By lemma \ref{L:simp.lin.ineqs} again, if $m_{i} < N$ the lowest dimensional affine plane, $\Pi_{i}$, 
on which $\sigma_{i}$ lies is given 
by $\Pi_{i} = \bigl\{ x_{i} \in \RR^{N} : x_{i} B_{i} = z_{i} \bigl\}$, where $B_{i}$ and $z_{i}$ are as in \eqref{E:sigma.i.ineqs.eqs}. (If $m_{i} = N$ then obviously $\Pi_{i} = \RR^{N}$.)

Clearly, $c$ is bounded. Hence, the set $c$ is a ``cell'' in the sense 
of Munkres \cite[Definition 7.2, p.\ 71]{jrM66}. There are finitely many such cells in $|K|^{n}$. 
$|K|^{n}$ is the union of them. The lowest dimensional plane containing $c$ 
is $\Pi := \Pi_{1} \times \cdots \times \Pi_{n}$ consisting of all points $(x_{1}, \ldots, x_{n}) \in \RR^{nN}$ with $x_{i} \in \RR^{N}$ 
s.t.\ $x_{i} B_{i} = z_{i}$ ($i \in \NN_{n}$). Therefore, as defined on Munkres \cite[p.\ 71]{jrM66}, 
    \begin{multline} \label{E:dim.of.prod.cell}
        \text{The dimension of } c \text{ is } \dim \Pi = \sum_{i=1}^{n} \dim \Pi_{i} 
          =  \sum_{i=1}^{n} (N - rank \, B_{i}) \\
            =  nN - \sum_{i=1}^{n} (N - m_{i}) = m_{1} + \cdots m_{n} = \dim \sigma_{1} + \cdots \dim \sigma_{n}. 
    \end{multline} 

Let $\dim c > 0$. Then it has faces (Munkres \cite[Definition 7.4, p.\ 73]{jrM66}). \emph{Claim:} 
    \begin{equation}  \label{E:faces.are.simp.prods}
      \text{Each face of } c \text{ is also a product of simplices in } K .
    \end{equation}
To see this, let $d$ be a face of $c$. By Munkres \cite[Lemma 7.5, p.\ 73]{jrM66}, 
$d$ is obtained by replacing some components of the vector inequalities $x_{i} A_{i} \, (I_{m_{i}}, \, - 1^{m_{i}}) \geq (y_{i}, -y_{i} 1^{m_{i}} - 1)$ ($i \in \NN_{n}$) in \eqref{E:sigma.i.ineqs.eqs} by equalities. (This only makes sense if $m_{i} > 0$.) 
First, an informal argument. For each relevant $i \in \NN_{n}$, drop the corresponding columns of $A_{i}$, 
replace $m_{i}$ in $(I_{m_{i}}, \, - 1^{m_{i}})$ by a smaller integer, drop components 
of $(y_{i}, -y_{i} 1^{m_{i}} - 1)$, and \emph{compensate for this} by appropriately adding columns to $B_{i}$ and entries to $z_{i}$. (If $m_{i} = N$ one has to create appropriate $B_{i}$ and $z_{i}$.) But by lemma \ref{L:simp.lin.ineqs} again, the resulting system of equalities and inequalities defines a new simplex. 

To make this precise, temporarily drop the subscript $i$, let $k = 1, \ldots, m$, and suppose WLOG it is the first $k$ components of the inequality system $x A \, (I_{m}, \, - 1^{m}) \geq (y_{i}, -y_{i} 1^{m} - 1)$ that become equalities in this boundary cell. 
Write $A = ( C_{1}^{N \times k}, C_{2}^{N \times (m-k)} )$ and $y = (y_{1}^{1 \times k}, y_{2}^{1 \times (m-k)})$. Then, by assumption, for $x$ in this boundary cell, 
    \begin{equation*}
      x C_{1} = y_{1} \text{ and } x C_{2} \geq y_{2} .
    \end{equation*}
Thus, $C_{1}$ has rank $k \leq m \leq N$ and therefore 
$(C_{1}^{T} C_{1})^{k \times k}$ has full rank $k$. 
Let $(C_{2}')^{N \times (m-k)} := C_{2} - C_{1} (C_{1}^{T} C_{1})^{-1} C_{1}^{T} C_{2}$. Note that $C_{2}'$ has full rank $m-k$. For suppose not. Then there is a nonzero vector $y^{(m-k) \times 1}$ s.t.\ $C_{2}' y = 0$. That means that there is a linear combination of columns of $C_{1}$ that equals a nontrivial linear combination of columns of $C_{2}$. This contradicts the fact that $A$ has rank $m$. Note that, in fact, $C_{1}^{T} C_{2}' = 0$. 

If $m < N$, let $(B')^{N \times (N-m+k)} := (B, C_{1})$. If $k = m$, drop $A$. Suppose $k < m$. Let $(A' )^{N \times (m-k)}= C_{2}'$. Then 
    \begin{equation*}
        (B')^{T} A' = 
            \begin{pmatrix}
                B^{T} C_{2}' \\
                C_{1}^{T} C_{2}'
            \end{pmatrix}
        .
    \end{equation*}
We have already observed that $C_{1}^{T} C_{2}' = 0$ and we know that $B^{T} (C_{1}, C_{2}) = B^{T} A = 0$, 
i.e., $B^{T} C_{1} = 0$ and $B^{T} C_{2} = 0$, so
$B^{T} A' = B^{T} C_{2} - B^{T} C_{1} (C_{1}^{T} C_{1})^{-1} C_{1}^{T} C_{2} = 0$. Therefore, $(B')^{T} A' = 0$. 
If $m = N$ define $B' = C_{1}$. In that case we still get $(B')^{T} A' = 0$.

Then the relations $x A \, (I_{m}, \, - 1^{m}) \geq (y, -y 1^{m} - 1)$, with $x C_{1} = y_{1}$, and $x B = z$ are equivalent to 
    \begin{multline*}  \label{E:xA.equiv.x(C1,C2)}
      x B = z, \; x C_{1} = y_{1}, \;  xC_{2} \geq y_{2}, \text{ and } \\
        -y_{1} 1^{k} - x C_{2} 1^{m-k}  = -x C_{1} 1^{k} - x C_{2} 1^{m-k} 
          = - x A 1^{m} \geq - y_{1} 1^{k} - y_{2} 1^{m-k} -1 . \\
            \text{ So } - x C_{2} 1^{m-k} \geq - y_{2} 1^{m-k} -1
    \end{multline*} 
This is equivalent to 
    \begin{multline*}
      x B' = z' := (z, y_{1}) \text{ and } x C_{2}' (I_{m-k}, -1^{m-k}) = x C_{2} 
        - x C_{1} (C_{1}^{T} C_{1})^{-1} C_{1}^{T} C_{2} (I_{m-k}, -1^{m-k}) \\
           \geq (y_{2}, -y_{2} 1^{m-k} - 1) - y_{1} (C_{1}^{T} C_{1})^{-1} C_{1}^{T} C_{2} (I_{m-k}, -1^{m-k}) .
    \end{multline*}
Let $y_{2}' := y_{2} - y_{1} (C_{1}^{T} C_{1})^{-1} C_{1}^{T} C_{2}$. Then the preceding becomes
    \begin{equation*}
      x B' = z' \text{ and } A' (I_{m-k}, -1^{m-k}) \geq (y_{2}', -y_{2}' 1^{m-k} - 1) .
    \end{equation*}
Thus, by lemma \ref{L:simp.lin.ineqs}, the factor of the boundary cell corresponding to a given $i$ is a simplex.
This completes the proof of the claim \eqref{E:faces.are.simp.prods}. Since the simplex factor, call it $\zeta_{i}$, 
just constructed depends only on $A_{i}$, $B_{i}$, etc., it is clear that if $s \in S$ then 
    \begin{equation} \label{E:permuting.zetas}
      g_{s}(\zeta_{1} \times \cdots \times \zeta_{n}) = (\zeta_{i_{s(1)}} \times \cdots \times \zeta_{i_{s(n)}})
    \end{equation}

Suppose $\dim c  = 0$. Then $c$ is a single point and already a simplex. Let $P^{0}$ be the collection of all these 0 simplices. 
With $m=0$, $P^{m}$ has the following properties.
    \begin{enumerate}
        \item $P^{m}$ is a finite simplicial complex. \label{I:Pm.is.simp.cmplx}
        \item $\dim P^{m} \leq m$. \label{I:dim.Pm.=.m}
        \item If $s \in S$ and $\tau \in P^{m}$ then $g_{s}(\tau)$ is a simplex of $P^{m}$ 
          and $\dim g_{s}(\tau) = \dim \tau$.
          \label{I:g.tau.has.same.dim}
        \item If $s \in S$ and $\tau \in P^{m}$ the images of the vertices of $\tau$ under $g_{s}$ 
          span $g_{s}(\tau)$.
          \label{I:images.of.verts.span.g(tau)}
        \item Each cell $c := \sigma_{1} \times \cdots \times \sigma_{n} 
          \subset \RR^{nN}$ in $|K|^{n}$ 
        of dimension $\leq m$ is the union of finitely many simplices in $P^{m}$.
          \label{I:K.n.union.of.simps}
    \end{enumerate}

Let $m = 0, 1, \ldots$ and suppose that we have constructed a finite simplicial complex $P^{m}$ of dimension $m$ having properties \ref{I:dim.Pm.=.m}, \ref{I:g.tau.has.same.dim}, \ref{I:images.of.verts.span.g(tau)}, and \ref{I:K.n.union.of.simps}. 

If the union of these simplices is $|K|^{n}$, we are done: Let $P = P^{m}$. 
Otherwise, by property \ref{I:K.n.union.of.simps}, $|K|^{n}$ contains a cell of the form \eqref{E:c.=.simplex.prod} 
of dimension $> m$. Let $c = \sigma_{1} \times \cdots \times \sigma_{n}$ be an arbitrary cell in $|K|^{n}$ of dimension $> m \geq 0$.
Since $\dim c > 0$, there exists $i$ s.t.\ $m_{i} := \dim \sigma_{i} > 0$. WLOG $i = 1$. Thus, by \eqref{E:sigma.i.ineqs.eqs} again, 
if $x_{1} \in \sigma_{1}$, the inequalities $x_{1} A_{1} \, (I_{m_{1}}, \, - 1^{m_{1}}) \geq (y_{1}, -y_{1} 1^{m_{1}} - 1) $ are satisfied. Replace $\sigma_{1}$ with one of its $(m_{1}-1)$-faces, $\tau$. Since, by lemma \ref{L:simp.lin.ineqs},
 the interior of $\sigma_{1}$ consists of precisely those $x_{1} \in \sigma_{1}$ s.t.\ the inequalities just mentioned are strict, for points of $\tau$, at least one of the inequalities in the system must be an equality. Therefore, as above, we can drop that equality provided with augment (or bring into existence) $B_{1}$. This increases the rank of $B_{1}$ by exactly 1 (because $\dim \tau = m_{1} - 1$). By \eqref{E:dim.of.prod.cell}, the dimension of the cell 
$c' := \tau \times \sigma_{2} \times \cdots \times \sigma_{n}$ is one less than that of $c$. Replace $c$ by $c'$. Continue until one obtains a cell of dimension $m+1$. Continue to use ``$c$'' to denote that cell. 

By Munkres \cite[Lemma 7.3, p.\ 72]{jrM66}, the boundary of $c$, call it $\text{Bd} \, c$, is the union of finitely many cells (faces) of dimension $m= \dim c - 1$. We have already proved that each such face is the product of simplices in $K$. Therefore, by property \ref{I:K.n.union.of.simps} above, $\text{Bd} \, c$ is the union of simplices in $P^{m}$. Let $L(c)$ be the collection of all simplices in $P^{m}$ lying in $\text{Bd} \, c$. $L(c)$ is a subcomplex of $P^{m}$ of dimension $m$ and $|L(c)| = \text{Bd} \, c$. 

Write $c := \sigma_{1} \times \cdots \times \sigma_{n}$, where $\sigma_{1}, \cdots, \sigma_{n} \in K$. If $\hat{\sigma}_{i}$ is the barycenter of $\sigma_{i}$ ($i \in \NN_{n}$; see \eqref{E:barycenter.of.sigma}), then, by \eqref{E:criterion.for.int.simp} and \eqref{E:barycenter.of.sigma}, 
$\hat{\sigma}_{i}$ is an interior point of $\sigma_{i}$. As observed just after \eqref{E:criterion.for.int.simp}, $\hat{\sigma}_{i}$ lies in the topological interior of $\sigma$ as a subspace of the plane, call it $\Pi_{i}$. 

Let $\Pi$ be the lowest dimensional plane in $\RR^{N}$ containing $c$. 
Then, as above, $\Pi = \Pi_{1} \times \cdots \times \Pi_{n}$. Since $\hat{\sigma}_{i}$ lies in the topological interior of $\sigma_{i}$ as a subset 
of $\Pi_{i}$ ($i \in \NN_{n}$), it must be the case that $z := z(c): = ( \hat{\sigma}_{1}, \ldots, \hat{\sigma}_{n} ) \in \RR^{nN}$ lies in the topological interior of $c$ as a subset of $\Pi$. 
	  
Since $c$ is convex, we may apply Munkres \cite[Lemma (8.1), p.\ 44]{jrM84} to conclude that $z \ast L(c)$ is a finite complex 
and $| z \ast L(c) | = c$. Since $\dim L(c) \leq m$, we have that $\dim \bigl[ z \ast L(c) \bigr] \leq m+1$. 
 Let $P^{m+1} := \bigcup_{c} z(c) \ast L(c)$, where the union is taken over all cells $c = \sigma_{1} \times \cdots \times \sigma_{n}$ of dimension $m+1$. Then $P^{m+1}$ has property \ref{I:dim.Pm.=.m} (with ``$m$'' replaced by ``$m+1$'', of course). 
 
Since $c$ was initially chosen to be an arbitrary cell in $|K|^{n}$ of dimension $> m$, $P^{m+1}$ has property \ref{I:K.n.union.of.simps}. 

Let $s \in S$ and $g := g_{s}$. Now, 
    \begin{equation*} \label{E:effect.of.g.on.cell}
        g_{s}(\sigma_{1} \times \sigma_{2} \times \cdots \times \sigma_{n}) 
          = \sigma_{i_{s(1)}} \times \sigma_{i_{s(2)}} \times \cdots \times \sigma_{i_{s(n)}}.
    \end{equation*}
Thus, $g(c) \subset |K|^{n}$ is another cell of the same dimension. Therefore, 
$z \bigl[ g(c) \bigr] = g \bigl[ z(c) \bigr]$ and, by \eqref{E:permuting.zetas}, $g \bigl[ L(c) \bigr] = L \bigl[ g(c) \bigr]$. Therefore, by the induction hypothesis, $P^{m+1}$ has properties \ref{I:g.tau.has.same.dim} and \ref{I:images.of.verts.span.g(tau)} as well. Finally, if $s \in S$, $v_{1}, \ldots, v_{p} \in \RR^{nN}$, and $\beta_{1}, \ldots, \beta_{p} \in \RR$ then obviously
	\begin{equation*}   
		g_{s} \left( \sum_{i=0}^{p} \beta_{i} v_{i} \right) 
		  = \sum_{i=0}^{p} \beta_{i} \, g_{s}(v_{i}).
	\end{equation*}
Thus, by property \ref{I:images.of.verts.span.g(tau)}, we have that \eqref{E:simp.map.commutes.with.bary} holds. It follows that $g$ is a simplicial homeomorphism of $|P|$ onto itself. Take $f : |P| \to |K|^{n}$ to be the identity. This completes the proof of the lemma.
  \end{proof}   

	\begin{lemma}   \label{L:invrs.is.subcmplx}
Let $f$ be a simplicial map from a complex $P$ to a complex $L$. If $\rho \in L$, then $f^{-1}(\rho) = |K_{1}|$, where $K_{1}$ is a, possibly empty, subcomplex of $P$. 
	\end{lemma}
	\begin{proof} 
Let $\rho \in L$. If $f^{-1}(\rho) = \varnothing$, we are done. So suppose $f^{-1}(\rho) \neq \varnothing$. By \eqref{E:x.in.exctly.1.simplex.intrr}, it suffices to show the following:  
	\begin{equation}  \label{E:in.for.dime.in.for.dollar}
		\text{If } \sigma \in P, x \in \text{Int} \, \sigma, \text{ and } f(x) \in \rho, 
		   \text{ then } \sigma \subset f^{-1}(\rho).
	\end{equation}  
For then we can just take
	\[
		K_{1} := \bigl\{ \sigma \in P : (\text{Int} \, \sigma) \cap f^{-1}(\rho) \neq \varnothing \bigr\}.
	\]
First, obviously $K_{1}$ has property \eqref{E:intersection.of.simps} because $P$ does. Second, if $\sigma \in K_{1}$ and $\tau$ is a face of $\sigma$, then, by \eqref{E:in.for.dime.in.for.dollar}, $f(\tau) \subset \rho$, so $\tau \in K_{1}$ and \eqref{E:complex.contains.all.faces} holds for $K_{1}$. Thus, $K_{1}$ is a subcomplex of $P$. If \eqref{E:in.for.dime.in.for.dollar} holds then obviously $|K_{1}| = f^{-1}(\rho)$.

We prove \eqref{E:in.for.dime.in.for.dollar}. Suppose $\sigma \in P$, $x \in \text{Int} \, \sigma$, and $f(x) \in \rho$. Thus, $f(x) \in \text{Int} \, \rho'$ for some face $\rho'$ of $\rho$. Let $v_{0}, \ldots, v_{n}$ be the vertices of $\sigma$. Then there exist $\beta_{0}, \ldots, \beta_{n}$, nonegative and summing to 1, s.t.\ $x = \sum_{i=0}^{n} \beta_{i} v_{i}$. Since $x \in \text{Int} \, \sigma$, by \eqref{E:criterion.for.int.simp}, all the $\beta$'s are strictly positive. By \eqref{E:simp.map.commutes.with.bary}, since $f$ is simplicial,
	\begin{equation}   \label{E:f(x).in.bary.coords}
		f(x) = \sum_{i=0}^{n} \beta_{i} f(v_{i}).
	  \end{equation}
At the same time, since $f$ is simplicial, $f(v_{i})$ ($i=0, \ldots, n$) are vertices of some $\tau \in L$. Since all the $\beta$'s are strictly positive, by \eqref{E:f(x).in.bary.coords} and \eqref{E:criterion.for.int.simp} again, $f(x)$ lies in the simplicial interior of $\tau$. Thus, $f(x) \in (\text{Int} \, \rho') \cap (\text{Int} \, \tau)$. Hence, by (\ref{E:intersection.of.simps}'), $\tau = \rho'$. In particular, $f(v_{i})$ ($i=0, \ldots, n$) lie in $\rho$. \eqref{E:in.for.dime.in.for.dollar} follows.
	\end{proof}

	\begin{lemma}  \label{L:vertices.in.sd}
Let $P$ be a simplicial complex and let $P' := \sd P$ be its first barycentric subdivision (Munkres, \cite[pp.\ 85--86]{jrM84}). We have the following.
 \begin{enumerate}
    \item Let $\rho \in P'$. Then the vertices of $\rho$ have the form $\hat{\sigma}_{i}$ ($i = 0, \ldots, n$; see \eqref{E:barycenter.of.sigma}), where $\sigma_{0}, \ldots, \sigma_{n} \in P$ and $\sigma_{0} \succ \cdots \succ \sigma_{n}$. (I.e., for $i \in \NN_{n}$, the simplex $\sigma_{i}$ is a proper face of $\sigma_{i-1}$.) Conversely, if $\sigma_{0}, \ldots, \sigma_{n} \in P$ and $\sigma_{0} \succ \cdots \succ \sigma_{n}$ then $\hat{\sigma}_{i}$ ($i = 0, \ldots, n$) span a simplex in $P'$. In particular, $\hat{\sigma}_{i}$ ($i = 0, \ldots, n)$ are geometrically independent. Moreover, $\rho \subset \sigma_{0}$. \label{I:sdP.and.sigma.hats}
   \item Let $\rho \in P'$ and suppose $\zeta_{0}, \ldots, \zeta_{n} \in P$ and $\rho = \langle \hat{\zeta}_{0}, \ldots, \hat{\zeta}_{n} \rangle$. Nothing is assumed concerning which, if any, of the $\zeta_{i}$'s are faces of other $\zeta_{i}$'s. Pick $\sigma_{0}, \ldots, \sigma_{k} \in P$ and $\sigma_{0} \succ \cdots \succ \sigma_{k}$ s.t.\ $\rho = \langle \hat{\sigma}_{0}, \ldots, \hat{\sigma}_{k} \rangle$. Then $k = n$ and, reordering if necessary, we must have $\zeta_{i} = \sigma_{i}$ ($i \in \NN_{n}$). 
\label{I:sigma.hat.verts.are.unique}
  \item If $P$ is a finite complex then given a metric on $|P|$ and given $\epsilon > 0$, there exists $N = 0, 1, 2, \ldots$ s.t.\ every simplex in $\text{sd}^{N} \, P$ (the complex that results from recursively applying the barycentric subdivision operator $N$ times) has diameter less than $\epsilon$. 
      \label{I:sd.makes.small}
 \end{enumerate}
	\end{lemma}
	\begin{proof} 
 Statement \ref{I:sdP.and.sigma.hats} is mostly Munkres \cite[Lemma 15.3, p.\ 86]{jrM84}. That $\rho \subset \sigma_{0}$ is trivial: $\sigma_{0} \succ \cdots \succ \sigma_{n}$ implies that all vertices of $\rho = \hat{\sigma}_{0} \succ \cdots \succ \hat{\sigma}_{n}$ lie in $\sigma_{0}$.

We prove statement \ref{I:sigma.hat.verts.are.unique}. Let $\rho \in P'$ and suppose $\zeta_{0}, \ldots, \zeta_{n} \in P$ and $\rho = \langle \hat{\zeta}_{0}, \ldots, \hat{\zeta}_{n} \rangle$. In particular, $\hat{\zeta}_{0}, \ldots, \hat{\zeta}_{n}$ are geometrically independent. In particular, $\zeta_{0}, \ldots, \zeta_{n}$ are distinct. By the first part of the lemma, we have $\rho = \langle \hat{\sigma}_{0}, \ldots, \hat{\sigma}_{k} \rangle$, where $\sigma_{0}, \ldots, \sigma_{k} \in P$ and $\sigma_{0} \succ \cdots \succ \sigma_{k}$. 

By \eqref{E:simplex.determines.verts}, $n=k$ and $\{ \hat{\zeta}_{0}, \ldots, \hat{\zeta}_{n} \} = \{ \hat{\sigma}_{0}, \ldots, \hat{\sigma}_{k} \}$. WLOG $\hat{\zeta}_{j} = \hat{\sigma}_{j}$ ($j = 0, \ldots, k$). By (\ref{E:intersection.of.simps}'), $\zeta_{j} = \sigma_{j}$. ($j = 0, \ldots, k$).

Statement \ref{I:sd.makes.small} is just Munkres \cite[Theorem 15.4, p.\ 86]{jrM84}.
	\end{proof}

We have the following. Recall, from the beginning of this appendix, what it means for vertices to ``span'' a simplex.
  \begin{lemma}  \label{L:geom.indep.f(v).means.geom.indep.f(x)}
Let $P$ and $Q$ be finite simplicial complexes and suppose $f : |P| \to |Q|$ is simplicial. Suppose further that if $v_{0}, \ldots, v_{p}$ span a simplex in $P$ then $f(v_{0}), \ldots, f(v_{p})$ span a simplex in $Q$. (E.g., $Q = P$ and $f$ is a simplicial homeomorphism of $P$ onto itself.) In particular, $f(v_{0}), \ldots, f(v_{p})$ are geometrically independent. Let $\sigma \in P$. Then	\begin{multline}  \label{E:f.sigma.preserves.geom.indep}
		\text{If } x_{0}, \ldots, x_{k} \in \sigma 
		         \text{ are geometrically independent if and only if } \\
			f(x_{0}), \ldots, f(x_{k}) \in |Q| \text{ are geometrically independent.}
	\end{multline}
  \end{lemma}
  \begin{proof}
 Suppose $x_{0}, \ldots, x_{k} \in \sigma$. For $i = 0, \ldots, k$ there exist $\beta_{i0}, \ldots, \beta_{ip} \geq 0$ s.t.\ 
	\begin{equation*}
		\sum_{j=0}^{p} \beta_{ij} = 1 \text{ and } x_{i} = \sum_{j=0}^{p} \beta_{ij} v_{j}.
	\end{equation*}

First, suppose $x_{0}, \ldots, x_{k}$ are geometrically independent but $f(x_{0}), \ldots, f(x_{k})$ are not. Then there exist $t_{0}, \ldots, t_{k} \in \RR$, not all 0, s.t.\ 
	\begin{equation*} 
		\sum_{i=0}^{k} t_{i} = 0,\text{ and } \sum_{i=0}^{k} t_{i} f(x_{i}) = 0.
	\end{equation*}
By \eqref{E:simp.map.commutes.with.bary}, we also have 
$f(x_{i}) = \sum_{j=0}^{p} \beta_{ij} f(v_{j})$. Thus,  
	\begin{equation*}
		0 = \sum_{i=0}^{k} t_{i} \; f \left( \sum_{j=0}^{p} \beta_{ij} v_{j} \right) 
		  =  \sum_{i=0}^{k} t_{i} \sum_{j=0}^{p} \beta_{ij} f(v_{j})
		    = \sum_{j=0}^{p} \left( \sum_{i=0}^{k} t_{i} \beta_{ij} \right) f(v_{j}).
	\end{equation*}
But
	\begin{equation*}
		\sum_{j=0}^{p} \left( \sum_{i=0}^{k} t_{i} \beta_{ij} \right) 
		    = \sum_{i=0}^{k} t_{i} \left( \sum_{j=0}^{p} \beta_{ij} \right) 
		      = \sum_{i=0}^{k} t_{i} = 0.
	\end{equation*}
By assumption, $f(v_{0}), \ldots, f(v_{p})$ are geometrically independent. Thus, 
	\begin{equation*}
		\sum_{i=0}^{k} t_{i} \beta_{ij} = 0, \quad j = 0, \ldots, p.
	\end{equation*}
Hence,
	\begin{equation*}
		0 = \sum_{j=0}^{p} \sum_{i=0}^{k} t_{i} \beta_{ij} v_{j} 
			= \sum_{i=0}^{k}  t_{i} \left( \sum_{j=0}^{p} \beta_{ij} v_{j} \right) 
			     = \sum_{i=0}^{k}  t_{i} x_{i}.
	\end{equation*}
Thus, $x_{0}, \ldots, x_{k}$ are geometrically dependent. Contradiction. Therefore, $f(x_{0}), \ldots, f(x_{k}) \in \sigma$ are geometrically independent.

Conversely, suppose $f(x_{0}), \ldots, f(x_{k})$ are geometrically independent but $x_{0}, \ldots, x_{k}$ are not. Then there exist $t_{0}, \ldots, t_{k} \in \RR$, not all 0, s.t.\ 
	\begin{equation*} 
		\sum_{i=0}^{k} t_{i} = 0,\text{ and } \sum_{i=0}^{k} t_{i} x_{i} = 0.
	\end{equation*}
Thus, 
	\begin{equation*}
		0 = \sum_{i=0}^{k} t_{i} \sum_{j=0}^{p} \beta_{ij} v_{j}
		    = \sum_{j=0}^{p} \left( \sum_{i=0}^{k} t_{i} \beta_{ij} \right) v_{j}.
	\end{equation*}
Since $v_{0}, \ldots, v_{p}$ are geometrically independent. We must have
	\begin{equation*}
		\sum_{i=0}^{k} t_{i} \beta_{ij} = 0, \quad j = 0, \ldots, p.
	\end{equation*}
Hence,
	\begin{equation*}
		0 = \sum_{j=0}^{p} \sum_{i=0}^{k} t_{i} \beta_{ij} f(v_{j}) 
			= \sum_{i=0}^{k}  t_{i} \left( \sum_{j=0}^{p} \beta_{ij} f(v_{j}) \right) 
			     = \sum_{i=0}^{k}  t_{i} f(x_{i}).
	\end{equation*}
Thus, $f(x_{0}), \ldots, f(x_{k})$ are geometrically dependent. Contradiction. Therefore, $x_{0}, \ldots, x_{k}$ are geometrically independent. This proves the claim \eqref{E:f.sigma.preserves.geom.indep}.
  \end{proof}
  
The following is probably already known. 

Recall the definition, \ref{D:simplicial.map}, of simplicial map.

    \begin{prop}  \label{P:P/G}
Let $P$ be a finite simplicial complex and let $G$ be a, necessarily finite, group of simplicial homeomorphisms of $P$ onto itself (definition \ref{D:simplicial.homeom}). Then there is a subdivision, $P''$, of $P$, a finite simplicial complex, $L$, and a simplicial map $f$ from $|P''|$ to $|L|$ with the following properties. 
	\begin{enumerate}
	\item $G$ is a group of simplicial homeomorphisms of $P''$ onto itself. 
	             \label{I:G.simplicial.on.sdP}
	\item If $\rho \in P''$, 
$x \in \rho$, and $g \in G$, then either $g(x) = x$ or $g(x) \notin \rho$.
  	\label{I:g.on.rho.equal.or.out}
	\item If $w_{0}, \ldots, w_{p} \in (P'')^{(0)}$ span a simplex in $P''$, $g_{0}, \ldots, g_{p} \in G$, and $g_{0}(w_{0}), \ldots, g_{p}(w_{p})$ span a simplex in $P''$, then there exists $h \in G$ s.t.\ $g_{i}(w_{i}) = h(w_{i})$ for $i = 0, \ldots, p$. \label{I:one.h.many.gs}
	\item The orbit space $|P''|/G$ is homeomorphic to $|L|$, we have $f \bigl( |P''| \bigr) = |L|$, and if $x,y \in |P''|$ then $f(x) = f(y)$ if and only if $Gy = Gx$, where $Gx := \bigl\{ g(x) : g \in G \bigr\}$. 
		\label{I:P/G.is.a.complex}
	\item If $v_{0}, \ldots, v_{p}$ span a simplex in $P''$ then $f(v_{0}), \ldots, f(v_{p})$ span a simplex in $L$. In particular, $f(v_{0}), \ldots, f(v_{p})$ are geometrically independent. If $\sigma \in P''$, then $f( \text{Int} \, \sigma) =  \text{Int} \, f(\sigma)$. \label{I:f.vs.geom.indep}   
	\item Let $\zeta \in L$, $\sigma \in P''$ and suppose $f^{-1}(\zeta) \cap (\text{Int} \, \sigma) \neq \varnothing$. Then $f(\sigma)$ is a face of $\zeta$. If $f^{-1}(\text{Int} \, \zeta) \cap (\text{Int} \, \sigma) \neq \varnothing$ then $f(\sigma) = \zeta$. \label{I:faces.of.zeta.in.L}
	\item Let  $\rho, \tau \in P''$. Then $f(\rho) = f(\tau)$ if and only if  there exists  
	  $g \in G$ s.t.\  $\rho = g(\tau)$. We have $f(\rho) \cap f(\tau) \neq \varnothing$ if and only if  there exists $g \in G$ s.t.\  $\rho \cap g(\tau) \neq \varnothing.$  \label{I:f(rho).f(tau).dsjnt.iff.rho.Gtau.dsjnt}
	\item If $\omega \in L$ then there exists $\rho \in P''$ s.t.\ $f(\rho) = \omega$. For any such $\rho$ we have $\dim \rho = \dim \omega$.  
	      \label{I:f(rho)=omega}
	\end{enumerate}
    \end{prop}
\begin{proof} 
 The subdivision $P''$ is just the ``second barycentric subdivision'', $P'' := \text{sd}^{2} \, P$, of $P$ (Munkres \cite[p.\ 86]{jrM84}). But to start with consider the first barycentric subdivision $P' := \sd P$. 

By assumption, $G$ is a group of simplicial homeomorphisms of $P$ onto itself. \emph{Claim:}  
	\begin{equation}  \label{E:group.of.simp.homeoms.on.P'}
		G \text{ is also a group of simplicial homeomorphisms of } P' \text{  onto itself.}
	\end{equation} 
Let $\sigma \in P$ and let $g \in G$. Since $G$ is a group, $g$ is one-to-one. Since $g$ is simplicial, this means
	\begin{multline}  \label{E:g.image.spans}
		\text{If } v_{0}, \ldots, v_{p} \in P^{(0)} \text{ span a simplex in } P, \text{ then }  \\
		    g(v_{0}), \ldots, g(v_{p}) \in P^{(0)} \text{ span a simplex in $P$ as well.}
	\end{multline}
Thus, if $\sigma_{0}, \ldots, \sigma_{n} \in P$ and $g \in G$, then $\sigma_{0} \succ \cdots \succ \sigma_{n}$ if and only if $g(\sigma_{0}) \succ \cdots \succ g(\sigma_{n})$. In addition (see \eqref{E:barycenter.of.sigma} and \eqref{E:simp.map.commutes.with.bary}),  
	\begin{equation} \label{E:g.hat.=.hat.g}
		g(\hat{\sigma}) = \widehat{g(\sigma)} \text{ and } \dim g(\sigma) = \dim \sigma,
		         \quad \sigma \in P.
	\end{equation}
Therefore, by lemma \ref{L:vertices.in.sd}\eqref{I:sdP.and.sigma.hats}, $\hat{\sigma}_{i}$ ($i = 0, \ldots, n$) span a simplex in $P'$ if and only if $g(\hat{\sigma}_{i})$ ($i = 0, \ldots, n$) do. To complete the proof that $G$ is a group of simplicial homeomorphisms on $P'$, we must prove the analogue of \eqref{E:simp.map.commutes.with.bary}. Suppose $\sigma_{0}, \ldots, \sigma_{n} \in P$ and $\sigma_{0} \succ \cdots \succ \sigma_{n}$. Then $\hat{\sigma}_{0}, \hat{\sigma}_{1} \ldots \hat{\sigma}_{n} \in\sigma_{0}$. The claim, \eqref{E:group.of.simp.homeoms.on.P'}, now easily follows from \eqref{E:barycenter.of.sigma}, \eqref{E:g.hat.=.hat.g}, and the fact that the elements of $G$ are simplicial homeomorphisms of $P$ onto itself. Applying this fact with $P'$ in place of $P$, \textbf{point (\ref{I:G.simplicial.on.sdP})} of the  proposition follows.

 Let $\sigma_{0} \succ \cdots \succ \sigma_{n}$ be simplices in $P$ and let $\rho \in P'$ be spanned by $\hat{\sigma}_{i}$ ($i = 0, \ldots, n)$, so, by lemma \ref{L:vertices.in.sd}\eqref{I:sdP.and.sigma.hats}, $\rho \subset \sigma_{0}$. Let $g \in G$, 
 and let $i = 0, \ldots, n$. \emph{Claim:} 
	\begin{equation}  \label{E:g.sigma.hat.in.rho?}
		\text{Either } g(\hat{\sigma}_{i}) = \hat{\sigma}_{i} 
		  \text{ or } g(\hat{\sigma}_{i}) \notin \rho. 
	\end{equation}
By \eqref{E:g.hat.=.hat.g},
$g(\hat{\sigma}_{i}) = \hat{\tau}$ with $\tau := g(\sigma_{i}) \in P$. Suppose 
$\hat{\tau} \in \rho \subset \sigma_{0}$. Thus, $(\text{Int} \, \tau) \cap \sigma_{0} \neq \varnothing$ so, by \eqref{E:Int.rho.cuts.sigma.then.rho.in.sigma}, $\tau$ is a face of $\sigma_{0}$. Therefore, $\{ \hat{\tau} \} \subset \rho$ is a 0-simplex in $P'$. Since $P'$ is a simplicial complex and $\hat{\tau} \in \rho $, it follows from \eqref{E:intersection.of.simps} that $\hat{\tau}$ must be one of the vertices $\hat{\sigma}_{0}, \ldots, \hat{\sigma}_{n}$ of $\rho$. Say $\hat{\tau} = \hat{\sigma}_{j}$. Thus, by lemma \ref{L:vertices.in.sd} statement \eqref{I:sigma.hat.verts.are.unique} (with $n=0$), for some $j = 0, \ldots, n$, we have $g(\sigma_{i}) = \tau = \sigma_{j}$. Since, by \eqref{E:g.hat.=.hat.g}, $\dim g(\sigma_{i}) = \dim \sigma_{i}$, we must have $j = i$. In particular, by \eqref{E:g.hat.=.hat.g} again, we have $g(\hat{\sigma}_{i}) = \widehat{g(\sigma_{i})} = \hat{\tau} = \hat{\sigma}_{i}$. This proves the claim. 

We prove \textbf{point (\ref{I:g.on.rho.equal.or.out})} of the proposition. Hypothetically, \emph{suppose} the following were true. Let $\rho \in P$ be spanned by $v_{0}, \ldots, v_{p} \in P^{(0)}$, let $g \in G$, and let $i = 0, \ldots, p$. Then
\begin{equation}  \label{E:g.v.in.rho?}
		\text{Either } g(v_{i}) =v_{i} \text{ or } g(v_{i}) \notin \rho.
	\end{equation}
(By \eqref{E:g.sigma.hat.in.rho?} and \eqref{E:group.of.simp.homeoms.on.P'}, we have that \eqref{E:g.v.in.rho?}  does hold if $P$ is a barycentric subdivision of a complex on which $G$ is a group of simplicial homeomorphisms.)

Now let $\rho \in P$ be spanned by $v_{0}, \ldots, v_{p} \in P^{(0)}$, let $x \in \rho$, let $g \in G$, and suppose 
$g(x) \neq x$. Write
	\begin{equation*}
		x = \sum_{i=0}^{p} \beta_{i} \, v_{i} \in \rho,
	\end{equation*}
where $\beta_{0}, \ldots,\beta_{p}$ are nonnegative and sum to 1. 
Let $0 \leq i_{0} < \cdots < i_{\ell} \leq \nvar$ be the indices, $i$, for which $\beta_{i} > 0$. 
Let 
	\[
		\tau \text{ be the simplex spanned by } g(v_{i_{k}}),
		\quad k = 0, \ldots, \ell.
	\] 
Then, by \eqref{E:g.image.spans}, we have $\tau \in P$. By \eqref{E:simp.map.commutes.with.bary} and \eqref{E:criterion.for.int.simp},
	\begin{equation} \label{E:g.x.in.tau}
		g(x) = \sum_{i=0}^{p} \beta_{i} \, g(v_{i})  \in \text{Int} \, \tau.
	\end{equation}  
Suppose $g(x) \in \rho$. Then, by \eqref{E:g.x.in.tau} and \eqref{E:Int.rho.cuts.sigma.then.rho.in.sigma}, $\tau$ is a face of $\rho$. That means that $g(v_{i_{k}}) \in \rho$ for $k = 0, \ldots, \ell$. But, since $g(x) \neq x$, for some $j = 0, \ldots, n$, we have $\beta_{j} > 0$ and $g(v_{j}) \neq v_{j}$.  Therefore, by \eqref{E:g.v.in.rho?}, we have $g(v_{j}) \notin \rho$. But for some $j = 0, \ldots, \ell$, we have $j = i_{k}$ since $\beta_{j} > 0$. Therefore, $g(v_{j}) \in \tau \subset \rho$ and $g(v_{j}) \notin \rho$.
Contradiction. It follows that $g(x) \notin \rho$, \emph{providing} \eqref{E:g.v.in.rho?} holds. But, by \eqref{E:group.of.simp.homeoms.on.P'}, we have that \eqref{E:g.sigma.hat.in.rho?} holds with $P'$ is place of $P$. Therefore, \eqref{E:g.v.in.rho?} holds with $P = P''$. \textbf{Point (\ref{I:g.on.rho.equal.or.out})} of the lemma follows.

We prove \textbf{point (\ref{I:one.h.many.gs})} of the lemma. Again, to start with assume \eqref{E:g.v.in.rho?} holds for simplices in $P$. Suppose $w_{0}, \ldots, w_{p} \in (P')^{(0)}$ span a simplex in $P'$, $g_{0}, \ldots, g_{p} \in G$, and $g_{0}(w_{0}), \ldots, g_{p}(w_{p})$ also span a simplex in $P'$. Permuting if necessary, we may assume, by lemma \ref{L:vertices.in.sd}\eqref{I:sdP.and.sigma.hats}, that $w_{i} = \hat{\sigma}_{i}$ ($i = 0, \ldots, p$), where $\sigma_{0}, \ldots, \sigma_{p} \in P$ and $\sigma_{0} \succ \cdots \succ \sigma_{p}$. Let $h_{i} = g_{0}^{-1} g_{i} \in G$ ($i = 0, \ldots, p$). Thus, $h_{0}$ is the identity and, by \eqref{E:g.image.spans} (with $P'$ in place 
of $P$ and $g = g_{0}^{-1}$), we have $h_{0}(w_{0}), \ldots, h_{p}(w_{p})$ span a simplex, $\omega$, in $P'$. I.e., by \eqref{E:g.hat.=.hat.g}, $\widehat{h(\sigma_{i})}$ ($i=1, \ldots, p$) span $\omega$. 

There exist $\tau_{0}  \succ \cdots \succ \tau_{\ell}$ in $P$, s.t.\
$\omega = \langle \hat{\tau}_{0}, \ldots, \hat{\tau}_{\ell} \rangle$. Hence, by lemma \ref{L:vertices.in.sd}\eqref{I:sigma.hat.verts.are.unique}, we have $\ell = p$ and the two (unordered) sets $\bigl\{ h(\sigma_{i}), \, i=1, \ldots, p \bigr\}$ and $\{ \tau_{i}, \, i=1, \ldots, p \}$ are the same. But, by \eqref{E:g.hat.=.hat.g}, $\dim h_{0}(\sigma_{0}) = \dim \sigma_{0} > \cdots > \dim \sigma_{p} = h_{p} (\sigma_{p})$. Thus, 
we have $h_{i}(\sigma_{i}) = \tau_{i}$ ($i = 0, \ldots, p$). Therefore, $h_{0}(\sigma_{0}) \succ \cdots \succ h_{p}(\sigma_{p}) $. 
In particular, since $h_{0}(\sigma_{0}) = \sigma_{0}$, we have 
	\begin{equation}  \label{E:hi.sigma.i.hat.in.rho}
		h_{i}(\hat{\sigma}_{i}) \in \sigma_{0} \quad (i = 0, \ldots, p).
	\end{equation}  

Let $i \in \NN_{n}$ and suppose 
	\begin{equation*}
		h_{i}(\hat{\sigma}_{i}) \ne \hat{\sigma}_{i}. 
	\end{equation*}
Let $v_{0}, \ldots,v_{p}$ be the vertices of $\sigma_{0}$ and let $v_{\ell_{0}}, \ldots, v_{\ell_{k}}$ be the vertices of $\sigma_{i}$. Then, by 
\eqref{E:g.hat.=.hat.g}, 
	\begin{equation*}
		h_{i}(\hat{\sigma}_{i}) = \tfrac{1}{k+1} \sum_{j=0}^{k} h_{i} (v_{\ell_{j}}) 
			\ne \tfrac{1}{k+1} \sum_{j=0}^{k} v_{\ell_{j}} = \hat{\sigma}_{i}.
	\end{equation*}
Thus, for some $t = 0, \ldots, k$, we must have $h_{i} (v_{\ell_{t}}) \ne v_{\ell_{t}}$. Then, by \eqref{E:g.v.in.rho?}, we have 
	\begin{equation}  \label{E:hi.v.lt.notin.sigma0}
		h_{i} (v_{\ell_{t}}) \notin \sigma_{0}.
	\end{equation} 
But, $h_{i} (v_{\ell_{j}})$ ($j=0, \ldots, k$) span $h_{i}(\sigma_{i}) \in P$. Moreover, by \eqref{E:g.hat.=.hat.g}, we have $h_{i}(\hat{\sigma}_{i}) = \widehat{h_{i}(\sigma_{i})} \in \text{Int} \, h_{i}(\sigma_{i})$. But, by \eqref{E:hi.sigma.i.hat.in.rho}, we have $h_{i}(\hat{\sigma}_{i}) \in \sigma_{0}$. Therefore, by \eqref{E:Int.rho.cuts.sigma.then.rho.in.sigma}, this means $h_{i}(\sigma_{i})$ is a face of $\sigma_{0}$. In particular,  $h_{i} (v_{\ell_{t}}) \in \sigma_{0}$. This contradicts \eqref{E:hi.v.lt.notin.sigma0}. Thus, 
$g_{0}^{-1} g_{i}(w_{i}) = h_{i}(w_{i}) = h_{i}( \hat{\sigma}_{i}) = \hat{\sigma}_{i} = w_{i} = g_{0}^{-1} g_{0}(w_{i})$. Hence, $g_{i}(w_{i}) = h(w_{i})$, where $h := g_{0}$, providing \eqref{E:g.v.in.rho?} holds. But by \eqref{E:g.sigma.hat.in.rho?}, we have that \eqref{E:g.v.in.rho?} holds with $P'$ in place of $P$. \textbf{Point (\ref{I:one.h.many.gs})} of the lemma follows.

Suppose $\omega_{0}, \ldots, \omega_{n} \in P'$ with $\omega_{0} \succ \cdots \succ \omega_{n}$, so $\hat{\omega}_{0}, \ldots, \hat{\omega}_{n}$ span a simplex in $P''$. 
\emph{Claim:} 
	\begin{equation} \label{E:G.hat.omegas.disjoint}
		\text{The sets } G \hat{\omega}_{i} := \{ g(\hat{\omega}_{i}) : g \in G \}
		           \; (i = 0, \ldots, n) \text{ are disjoint.}
	\end{equation}
Since $G$ is a group it suffices to show that for any $0 \leq i <  j \leq n$ we have $G \hat{\omega}_{j} \neq G \hat{\omega}_{i}$, i.e., to prove disjointedness it suffices to show inequality. Suppose that, on the contrary, there exist $0 \leq i <  j \leq n$ s.t.\ $G \hat{\omega}_{j} = G \hat{\omega}_{i}$. Since $G$ is a group it contains an identity element. Thus, $\hat{\omega}_{i} \in G \hat{\omega}_{j}$. Hence, there exists $g \in G$ s.t., by \eqref{E:g.hat.=.hat.g} and \eqref{E:bary.in.Int}, 
	\[
		\text{Int} \, g(\omega_{j} \ni \widehat{g(\omega_{j})} = g(\hat{\omega}_{j}) 
		             = \hat{\omega}_{i} \in \text{Int} \, \omega_{i}.
	\]
I.e., $\text{Int} \, \omega_{i} \cap \text{Int} \, g(\omega_{j}) \neq \varnothing$. Therefore, by 
(\ref{E:intersection.of.simps}'), $\omega_{i} = g(\omega_{j})$. But by \eqref{E:g.hat.=.hat.g} again, $\dim g(\omega_{j}) = \dim \omega_{j} < \dim \omega_{i}$. This contradiction proves the claim.

We just showed that if $\omega_{0}, \ldots, \omega_{n} \in P'$ with $\omega_{0} \succ \cdots \succ \omega_{n}$ then $G \hat{\omega}_{i}$ ($i = 0, \ldots, n$) are disjoint. Form the abstract simplicial complex (Munkres \cite[p.\ 15]{jrM84}), 
	\[
		P''_{G} := \bigl\{ \{ G \hat{\omega}_{0}, \ldots, G \hat{\omega}_{n} \} : 
		   \omega_{0}, \ldots, \omega_{n} \in P'
		      \text{ and } \omega_{0} \succ \cdots \succ \omega_{n}, 
		        \; n=0,1, \ldots,  \bigr\}.
	\]
Thus, $(P''_{G})^{(0)}$ consists of orbits of $G$. By Munkres \cite[Theorem 3.1, p.\ 15]{jrM84}, there is a simplicial complex, $L$, and a one-to-one  correspondence, 
$s : \bigl\{ Gv \in (P''_{G})^{(0)} : v \in (P'')^{(0)} \bigr\} \to L^{(0)}$ s.t.\ $\{ Gv_{0}, \ldots, Gv_{q} \} \in P''_{G}$ if and only if $s(Gv_{0}), \ldots, s(Gv_{q})$ span a simplex in $L$. But if $v_{0}, \ldots, v_{q} \in (P'')^{(0)}$ span a simplex in $P''$, then, by lemma \ref{L:vertices.in.sd}, there exist $\omega_{0}, \ldots, \omega_{q} \in P'$ with $\omega_{0} \succ \cdots \succ \omega_{q}$ s.t.\ $v_{i} =  \hat{\omega}_{i}$ ($i=0, \ldots, q$) up to permutation. Thus, by definition of $P''_{G}$, we have that $Gv_{0}, \ldots, Gv_{1}$ spans a simplex in $P''_{G}$. But that means $s(Gv_{0}), \ldots, s(Gv_{q})$ span a simplex in $L$. In short, $v_{0}, \ldots, v_{q} \in (P'')^{(0)}$ span a simplex in $P''$ means $s(Gv_{0}), \ldots, s(Gv_{q})$ span a simplex in $L$. Hence, by Munkres \cite[Lemma 2.7, p.\ 12]{jrM84}, there is unique simplicial map $f : |P''| = |P| \to |L|$ that takes each $v \in (P'')^{(0)}$ to $s(G v) \in L$. (Note that $f$ is surjective because $s$ is.)

Let $\sigma \in P''$ have vertices $v_{0}, \ldots, v_{p} \in (P'')^{(0)}$ so $f(v_{j}) = s(Gv_{j})$ ($j=0, \ldots, p$) span the simplex $f(\sigma) \in L$. (In particular, $f(v_{j})$, $j=0, \ldots, p$, are geometrically independent.) Let $x \in \text{Int} \, \sigma$. Then, by \eqref{E:criterion.for.int.simp}, there exist $\beta_{0}, \ldots, \beta_{p}$ all strictly positive, 
s.t.\  $x = \sum_{i=0}^{p} \beta_{i} v_{i}$. Since $f$ is simplicial, by \eqref{E:simp.map.commutes.with.bary}, 
$f(x) = \sum_{i=0}^{p} \beta_{i} s(Gv_{i}) \in \text{Int} \, f(\sigma)$. 
Conversely, suppose $z \in \text{Int} \, f(\sigma)$. Then, by \eqref{E:criterion.for.int.simp} again, there exist $\gamma_{0}, \ldots, \gamma_{p}$ all strictly positive, s.t.\  $z = \sum_{i=0}^{p} \gamma_{i} s(Gv_{i}) = f(x)$, 
where $x = \sum_{i=0}^{p} \gamma_{i} v_{i} \in \text{Int} \, \sigma$. \textbf{Point (\ref{I:f.vs.geom.indep})} follows.

Let $\zeta \in L$, $\sigma \in P''$ and let $x \in f^{-1}(\zeta) \cap (\text{Int} \, \sigma)$. Then by point (\ref{I:f.vs.geom.indep}), $f(\sigma) \in L$ and $f(x) \in \zeta \cap \bigl( \text{Int} \, f(\sigma) \bigr)$. It follows from \eqref{E:Int.rho.cuts.sigma.then.rho.in.sigma} that $f(\sigma)$ is a face of $\zeta$. Now suppose $f^{-1}(\text{Int} \, \zeta) \cap (\text{Int} \, \sigma) \neq \varnothing$. Then, by point (\ref{I:f.vs.geom.indep}) again, we have $(\text{Int} \, \zeta) \cap \bigl( \text{Int} \, f(\sigma) \bigr)  \neq \varnothing$. Hence, by (\ref{E:intersection.of.simps}'), $ \zeta= f(\sigma)$. This completes the proof of \textbf{point (\ref{I:faces.of.zeta.in.L})}.

Let $\pi : |P''| \to |P''|/G$ be the quotient projection. Since $|P''| = |P|$ is compact, $|P''|/G = \pi \bigl( |P''| \bigr)$ is compact. We show that  
	\begin{equation}  \label{E:f.factors.thru.P/G}
		\text{There is a continuous map } \phi : |P|/G \to |L| \text{ s.t.\ } \phi \circ \pi = f.
	\end{equation}  
Let $\eta \in |P''|/G$. There exists $x \in |P''|$ s.t.\ 
$\pi^{-1}(\eta) = Gx := \bigl\{ g(x) : g \in G \bigr\}$. We have 	\begin{equation*}
		x = \sum_{i=0}^{p} \beta_{\hat{\omega}_{i}} \, \hat{\omega}_{i},
	\end{equation*}
where $\beta_{\hat{\omega}_{0}}, \ldots,\beta_{\hat{\omega}_{p}}$ are nonnegative and sum to 1, $\omega_{0}, \ldots, \omega_{p} \in P'$, and $\omega_{0} \succ \cdots \succ \omega_{p}$. Thus,
 	\begin{equation*}
		\eta = Gx = \{ g(x) : g \in G \}
		    = \left\{ \sum_{i=0}^{p} \beta_{\hat{\omega}_{i}} \, g(\hat{\omega}_{i}) : 
		      g \in G \right\}.
	\end{equation*}
But, if $g \in G$, 
	\begin{equation*}
		f \bigl[ g(x) \bigr] = \sum_{i=0}^{p} \beta_{\hat{\omega}_{i}} \, 
		  f \bigl[ g(\hat{\omega}_{i}) \bigr]
			= \sum_{i=0}^{p} \beta_{\hat{\omega}_{i}} \, s (G \hat{\omega}_{i}).
	\end{equation*}
I.e., $f$ is constant on $\pi^{-1}( \eta )$. Therefore, from Munkres \cite[p.\ 112]{jrM84}, we see that \eqref{E:f.factors.thru.P/G} holds.

We show that $\phi : |P''|/G \to |L|$ is a homeomorphism. Now, $\phi : |P''|/G \to |L|$ is onto since $f$ is. Since, $|P''|/G$ is compact and $|L|$ is Hausdorff, by Simmons \cite[Theorem E, p.\ 131]{gfS63}, it suffices to show that $\phi$ is one-to-one. Let $x, y \in |P|$ be s.t.\ $z := \phi(Gx) = \phi(Gy)$. by \eqref{E:criterion.for.int.simp}, we may write 
	\[
		x = \sum_{i=0}^{p} \beta_{i} \, \hat{\omega}_{i},  \quad 
		   y = \sum_{i=0}^{q} \gamma_{i} \, \hat{\psi}_{i}, \quad 
		      z = \sum_{i=0}^{r} \, \delta_{i} w_{i}, 
	\]
where $\beta_{0}, \ldots,\beta_{p}$ are strictly positive and sum to 1, etc.; $\omega_{0}, \ldots, \omega_{p} \in P'$ and $\omega_{0} \succ \cdots \succ \omega_{p}$, $\psi_{0}, \ldots, \psi_{p} \in P'$ and $\psi_{0} \succ \cdots \succ \psi_{q}$; and $w_{0}, \ldots, w_{r} \in L^{(0)}$ span a simplex $\tau \in L$. Thus,
	\begin{equation}  \label{E:from.sigma.tau.to.z}
		\sum_{i=0}^{p} \beta_{i} \, s( G \hat{\omega}_{i}) = f(x) = \phi \circ \pi(x)
			= \phi(Gx) = z = \phi(Gy) = f(y)
		    = \sum_{i=0}^{q} \gamma_{i} \, s( G \hat{\psi}_{i} ).
	\end{equation}

By \eqref{E:G.hat.omegas.disjoint}, $G \hat{\omega}_{i}$ ($i=0, \ldots, p$) are distinct so they constitute a simplex in $P''_{G}$. Similarly for $G \hat{\psi}_{i}$ ($i=0, \ldots, q$). Therefore, since $s$ is one-to-one, we must have that the points $s( G \hat{\omega}_{i})$ ($i=0, \ldots, p$) span a simplex $\rho \in L$ and  $s( G \hat{\psi}_{j})$ ($j=0, \ldots, q$) span a simplex $\zeta \in L$. But by \eqref{E:from.sigma.tau.to.z},
	\[
		z \in (\text{Int} \, \rho) \cap (\text{Int} \, \zeta) \cap (\text{Int} \, \tau).
	\]
Hence, by (\ref{E:intersection.of.simps}'), $\rho = \zeta = \omega$. In particular, $p = r = q$. In addition, by uniqueness of barycentric coordinates, we may assume $s(G \hat{\omega}_{i}) =  s(G \hat{\psi}_{i})$ and $\beta_{i} = \delta_{i} = \gamma_{i}$ ($i=0, \ldots, r$). Since $s$ is injective, we have $G \hat{\omega}_{i} =  G \hat{\psi}_{i}$ ($i=0, \ldots, r$).

Therefore, there exist $g_{i} \in G$ ($i = 1, \ldots, r$) s.t.\ $\hat{\omega}_{i} = g_{i}(\hat{\psi}_{i})$ ($i=0, \ldots, r$). In particular, $g_{i}(\hat{\psi}_{i})$, ($i=0, \ldots, r$) spans a simplex in $P''$. By point (\ref{I:one.h.many.gs}) of the proposition, there exists $h \in G$ s.t.\ $g_{i}(\hat{\psi}_{i}) = h(\hat{\psi}_{i})$ for $i = 0, \ldots, r$. I.e., $\hat{\omega}_{i} = h(\hat{\psi}_{i})$. But then,
	\[
		x = \sum_{i=0}^{r} \delta_{i} \, \hat{\omega}_{i} 
		 	= \sum_{i=0}^{q} \delta_{i} \, h(\hat{\psi}_{i}) = h(y).
	\]
I.e., $Gx = Gy$ so $\phi$ is one-to-one. This proves \textbf{point (\ref{I:P/G.is.a.complex})}. 

We prove \textbf{point (\ref{I:f(rho).f(tau).dsjnt.iff.rho.Gtau.dsjnt})}. Suppose $\rho \in P''$, $\tau \in P''$, $g \in G$, and $\rho = g(\tau)$. Then $G \rho = G \tau$ so, by point (\ref{I:P/G.is.a.complex}) of the proposition,  
$f(\rho) = f(\tau)$. 
Conversely, suppose $f(\tau) = f(\rho) \neq \varnothing$. 
Let $z \in \text{Int} \, f(\rho) = \text{Int} \, f(\tau)$, $x \in f^{-1} \bigl( \{z\} \bigr) \cap \rho$, and $y \in f^{-1} \bigl( \{z\} \bigr) \cap \tau$. Thus, $f(x) = f(y)$. By point (\ref{I:f.vs.geom.indep}) of the proposition, we may assume $x \in \text{Int} \, \rho$ and $y \in \text{Int} \, \tau$. By point (\ref{I:P/G.is.a.complex}) of the proposition, 
there is $g' \in G$, s.t.\ $y = g'(x) \in \text{Int} \, g'(\rho)$. 
Thus, $\bigl( \text{Int} \, g'(\rho) \bigr) \cap (\text{Int} \, \tau) \neq \varnothing$. By 
(\ref{E:intersection.of.simps}') we must have $\tau = g'(\rho)$. 

Suppose we merely have $f(\rho) \cap f(\tau) \neq \varnothing$. Then by \eqref{E:intersection.of.simps},
the intersection $f(\rho) \cap f(\tau) \neq \varnothing$ is a simplex, $\zeta \in L$. 
Since $f$ is simplicial, there are faces $\rho'$ and $\tau'$ of $\rho$ and $\tau$, resp., s.t.\ $f(\rho') = \zeta = f(\tau')$. Hence, by the preceding paragraph, there exists $g \in G$ s.t.\ $\rho \supset \rho' = g(\tau') \subset g(\tau)$. I.e., $\rho \cap g(\tau) \neq \varnothing$. Conversely, suppose
$\rho \cap g(\tau) \neq \varnothing$. Now, $\rho \cap g(\tau) \neq \varnothing$ is a face, $\rho'$, of $\rho$ and of $g(\tau)$. Hence, $g^{-1}(\rho)$ contains a face, $\tau'$ of $\tau$ s.t.\ $\rho' = g(\tau')$. By the last paragraph, this means that $f(\rho') = f(\tau')$. I.e., $f(\rho) \cap f(\tau) \neq \varnothing$. This concludes the proof of \textbf{point (\ref{I:f(rho).f(tau).dsjnt.iff.rho.Gtau.dsjnt})}.

We prove \textbf{point (\ref{I:f(rho)=omega})}. Let $\omega \in L$ and let $w_{0}, \ldots, w_{q}$ be the vertices of $\omega$. By lemma \ref{L:vertices.in.sd}\eqref{I:sdP.and.sigma.hats} and definition of $L$, we have $w_{i} = s(G v_{i})$ ($i=0, \ldots, q$), where $v_{0}, \ldots, v_{q}$ span a simplex, $\sigma \in P''$. But $s(G v_{i}) = f(v_{i})$ ($i=0, \ldots, q$). Thus, $f(\sigma) = \omega$. Suppose $\rho \in P''$ and $f(\rho) = \omega$. (We just showed that such a $\rho$ exists.) By point (\ref{I:f.vs.geom.indep}), we have $\dim \rho = \dim \omega$.
\end{proof}

  \begin{corly}  \label{C:invar.subdivisions}
Let $P$ be a finite simplicial complex and let $g : |P| \to |P|$ be a simplicial homeomorphism from $P$ onto itself. Then there is an arbitrarily fine subdivision, $P'$ of $P$ s.t.\ $g : |P'| \to |P'|$ is a simplicial map $P'$ to itself. (``Arbitrarily fine'' means $P'$ can be chosen so that the maximum diameter of any simplex in $P'$ is arbitrarily small.) If $G$ is a group of simplicial homeomorphisms on $P$ onto itself then there is an arbitrarily fine subdivision, $P'$ of $P$ s.t.\ $g : |P'| \to |P'|$ is a simplicial map $P'$ to itself simultaneously for all $g \in G$.
  \end{corly}
  \begin{proof}  
From the proof of point (\ref{I:G.simplicial.on.sdP}) of the proposition we know that $G$ is a group of simplicial homeomorphisms from the first barycentric subdivision of $P$ onto itself. Now iterate and apply lemma \ref{L:vertices.in.sd}, statement \ref{I:sd.makes.small}. If $g : |P| \to |P|$ is a simplicial homeomorphism from $P$ onto itself, then just consider the group $G := \{ g, g^{-1}, identity \}$.
  \end{proof}

\chapter{Polyhedral Approximation}  \label{Chptr:polyhed.approx}
The following is proved in \cite{spE09.PolyhedralApproxLong} (original version 2011) and \cite{spE.PolyhedralApprox}. It is vaguely reminiscent of the ``push-out lemma'' of Fleming and Federer (Guth \cite[Lemma 0.5, p.\ 2]{lG2015.VolumesOfBalls}) from the 1950's. See also Guth \cite[Proposition 7.1, p.\ 26]{lG2013.ContractionOfAreas}. Let $X$ be a metric space 
and let $A, B \subset X$. Recall that the Hausdorff distance 
between $A$ and $B$ is 
    \begin{equation*}
      \max \left\{ \sup_{x \in A} dist(x,B), \sup_{x \in B} dist(x,A) \right\}.
    \end{equation*} 
Also recall \eqref{E:integer.part.floor}. Then we have 
   \begin{theorem}   \label{T:polyApproxThm}
     Let $P$ be a finite    
     simplicial complex lying in a Euclidean space.  
  	Let $|P|$ be the polytope or underlying space of $P$.  Use the metric on $|P|$ that it inherits 
	from the ambient Euclidean space.  Let $\Ss \subset |P|$ be closed.
     Let $\F$  be a topological space and suppose $\Phi : |P| \setminus \Ss \to \F$ is continuous.
        Let $Q$ be a subcomplex of $P$ (e.g., $Q = P$), let $a \geq 0$, and suppose 
          $\dim \bigl( \Ss \cap |Q| \bigr) \leq a$.  Then there is a closed set, 
          $\tilde{\Ss} \subset |P|$ 
          and a continuous map $\tilde{\Phi} : |P| \setminus \tilde{\Ss} \to \F$ such that:
            \begin{enumerate} 
              \item If $\F$ is a metric space and $\Phi$ is locally Lipschitz off $\Ss$ 
              then $\tilde{\Phi}$ 
              is locally Lipschitz off $\tilde{\Ss}$. \label{I:Phi.tilde.locally.Lip.if.Phi.is}   
              \item $\dim (\tilde{\Ss} \cap |Q|)  \leq \dim (\Ss \cap |Q|)$ 
              and $\dim \tilde{\Ss} \leq \dim \Ss$.
              	            \label{I:dim.Stilde.no.bggr.thn.dim.S}
              \item $\tilde{\Ss} \cap |Q|$ is either empty or the underlying space 
              of a subcomplex of the $\lfloor a \rfloor$-skeleton of $Q$.      
                        \label{I:S.tilde.subcomp}
              \item Suppose $\tau \in P$ has the following property.  If $\rho \in Q$ and 
               $(\text{Int} \, \rho) \cap \Ss \ne \varnothing$ then $\tau \cap \rho = \varnothing$.  
              Then $\tilde{\Ss} \cap \tau = \Ss \cap \tau$ and $\tilde{\Phi}$ and $\Phi$ agree 
              on $\tau \setminus \Ss$.   
                              \label{I:no.change.off.nbhd.of.S.in.Q}  
              \item Let $\rho \in P \setminus Q$. 
              (But $\rho \cap |Q| \neq \varnothing$ is possible.)  
              Then for every $s \geq 0$, if 
              $\Hm^{s} \bigl( \Ss \cap (\text{Int} \, \rho) \bigr) = 0$ then 
              $\Hm^{s} \bigl( \tilde{\Ss} \cap (\text{Int} \, \rho) \bigr)  = 0$.  
              In particular, $\dim \bigl( \tilde{\Ss} \cap (\text{Int} \, \rho) \bigr) 
                       \leq \dim \bigl( \Ss \cap (\text{Int} \, \rho) \bigr)$.   
                               \label{I:no.change.off.Q}
              \item If $\tau \in Q$ and 
              $\Hm^{\lfloor a \rfloor} \bigl( \tilde{\Ss} \cap (\text{Int} \, \tau) \bigr) > 0$, 
              then $\tau$ is an $\lfloor a \rfloor$-simplex and 
                 $\Hm^{\lfloor a \rfloor} \bigl( \Ss \cap (\text{Int} \, \sigma) \bigr) > 0$ 
                 for some simplex $\sigma$ of $Q$ having $\tau$ as a face.  
                 ($\sigma = \tau$ is possible.)  
                    \label{I:Ha.tau.=0.if.Ha.all.nbrs.0}
              \item If $y \in \tilde{\Ss}$ then there exists $\rho \in P$ such that $y \in \rho$ 
              and $\rho \cap \Ss \ne \varnothing$ and \emph{vice versa:} 
              If $y \in \Ss$ then there exists $\rho \in P$ such that $y \in \rho$ 
              and $\rho \cap \tilde{\Ss} \ne \varnothing$.
              Thus, the Hausdorff distance between $\Ss$ 
              and $\tilde{\Ss}$ does not exceed the largest of the diameters 
              of the simplices in $P$. 
	              \label{I:dist.tween.S.tilde.and.S}
	    \item If $\sigma \in P$ then 
	          $\tilde{\Phi}(\sigma \setminus \tilde{\Ss}) \subset \Phi(\sigma \setminus \Ss)$.
	              \label{I:Phi.tilde.sigma.subset.Phi.sigma}
              \item There is a constant $K < \infty$ depending only on $a$ and $Q$ such that
	                 \begin{equation}  \label{E:polyhedral.volume.magnification.factor}
	                      \Hm^{a} \bigl( \tilde{\Ss} \cap |Q| \bigr) 
	                        \leq K \Hm^{a} \bigl( \Ss \cap |Q| \bigr).
	                 \end{equation}    \label{I:bound.on.S.tilde.vol}
             \item There is a constant $K < \infty$ depending only on $a$ and $P$ 
             with the following property. For every $\epsilon > 0$ there is a subdivision, 
             $P'$, of $P$ such that $diam(\zeta) < \epsilon$ 
             for every $\zeta \in P'$ and parts 
      (\ref{I:Phi.tilde.locally.Lip.if.Phi.is}) through (\ref{I:Phi.tilde.sigma.subset.Phi.sigma})
      above and \eqref{E:polyhedral.volume.magnification.factor} 
      hold when $P$ is replaced by $P'$ and $Q$ is replaced by the corresponding 
      subcomplex of $P'$ (subdivision of $Q$). Moreover, given $q = 1, \ldots, p$, the same $K < \infty$ works for \emph{any} subcomplex $Q$ of dimension $q$. The 
$\epsilon = \infty$ case: There exists $K < \infty$ s.t., without subdividing $P$, \eqref{E:polyhedral.volume.magnification.factor} holds uniformly in subcomplexes $Q$ of dimension $q$. \label{I:arb.fine.subdivision}
            \end{enumerate}             
      \end{theorem}
 
  \begin{remark}[(co)homology of $\tilde{\Ss}$]  \label{E:S.tilde.(co)homology}
In \cite[remark 1.11]{spE09.PolyhedralApproxLong} it is speculated that $\tilde{\Ss}$ might inherit (co)homology from $\Ss$.   
  \end{remark}
  
Recall the definition of a simplicial homeomorphism (definition \ref{D:simplicial.homeom}).
      
  \begin{prop}  \label{P:Phi.invar.Phi.tilde.invar}
Let  $P$, $\F$, $\Phi$, $\Ss'$, $Q$, $\tilde{\Ss}$, $\tilde{\Phi}$, and $a \geq 0$ be as in theorem \ref{T:polyApproxThm}. (In particular $\tilde{\Ss}$ is closed.)  Suppose $G$ is a, necessarily finite, group of simplicial homeomorphisms of $P$ onto itself. Suppose for every $g \in G$ the restriction of $g$ to $|Q|$ is a simplicial homeomorphism of $Q$ onto itself. Furthermore, suppose $g(\Ss') = \Ss'$ for every $g \in G$ and 
$\Phi \bigl[ g(x) \bigr] = \Phi(x)$ for every $g \in G$ and $x \in |P| \setminus \Ss'$. Then, replacing $P$ (and correspondingly, $Q$) by a subdivision if necessary, we may assume $g(\tilde{\Ss}) = \tilde{\Ss}$ for every $g \in G$ 
and $\tilde{\Phi} \bigl[ g(x) \bigr] = \tilde{\Phi}(x)$ for every $g \in G$ and 
$x \in |P| \setminus \tilde{\Ss}$. Moreover, we may choose the subdivision $P'$ in part \ref{I:arb.fine.subdivision} of theorem \ref{T:polyApproxThm} so that $G$ is a group of simplicial is homeomorphisms of $P'$ onto itself.
  \end{prop}
Since $g(\Ss') = \Ss'$ and $\Phi \circ g = \Phi$ for every $g \in G$ we say that $\Ss'$ and $\Phi$ are ``$G$-invariant''.
      \begin{proof}
 Subdividing $P$ if necessary, we may assume $P'' = P$ has properties \eqref{I:G.simplicial.on.sdP} -- \eqref{I:f(rho)=omega} in proposition \ref{P:P/G}. Let $f$ and $L$ be as in proposition \ref{P:P/G}. It is convenient to consider a subdivision $L'$ of $L$. ($L' = L$ is possible.) The idea of the proof is to apply theorem \ref{T:polyApproxThm} to $L'$ and the obvious surrogate for $\Phi$ and then to lift to $P$, but there are a \emph{lot} of boring details.

Let $\sigma \in P$. Then by proposition \ref{P:P/G}(\ref{I:f.vs.geom.indep}), 
	\begin{equation*}
		\tau := f(\sigma) \in L. 
	\end{equation*}
Now, $\tau$ is the union of finitely many simplices in $L'$. Let $v_{0}, \ldots, v_{p} \in P^{(0)}$ be the vertices of $\sigma$. Suppose $\beta_{i}, \gamma_{i} \geq 0$ ($i = 0, \ldots, p$) and suppose $\sum_{i=0}^{p} \beta_{i} = 1 = \sum_{i=0}^{p} \gamma_{i}$. 
Let $x = \sum_{i=0}^{p} \beta_{i} v_{i} \in \sigma$ and $y = \sum_{i=0}^{p} \gamma_{i} v_{i} \in \sigma$. Suppose $f(x) = f(y)$. Now, by \eqref{E:simp.map.commutes.with.bary},
	 \[
		 0 = f(x) - f(y) = \sum_{i=0}^{p} (\beta_{i} - \gamma_{i}) f(v_{i}).
	 \]
Moreover, $ \sum_{i=0}^{p} (\beta_{i} - \gamma_{i}) = 0$. By proposition \ref{P:P/G}(\ref{I:f.vs.geom.indep}), $f(v_{0}), \ldots, f(v_{p})$ are geometrically independent since $v_{0}, \ldots, v_{p}$ are. Therefore, if $f(x) = f(y)$, then $\beta_{i} =\gamma_{i}$  ($i = 0, \ldots, p$), i.e., $x = y$. Therefore, 
	\begin{equation}   \label{E:f.sigma.is.1-1}
		\text{The restriction, } f_{\sigma} := f \restriction_{\sigma}, \text{ of $f$ to $\sigma$ is one-to-one.}
	\end{equation} 

By lemma \ref{L:geom.indep.f(v).means.geom.indep.f(x)} and proposition \ref{P:P/G}, point (\ref{I:f.vs.geom.indep}),
	\begin{multline}  \label{E:f.sigma.preserves.geom.indep}
		\text{If } x_{0}, \ldots, x_{k} \in \sigma 
		         \text{ are geometrically independent if and only if } \\
			f(x_{0}), \ldots, f(x_{k}) \text{ are geometrically independent.}
	\end{multline}

Let $\zeta \in L'$ lie in $\tau := f(\sigma)$ and let $z_{0}, \ldots, z_{q}$ be the vertices of $\zeta$. Let $u_{j} := f_{\sigma}^{-1}(z_{j})$ ($j=0, \ldots, q$). By \eqref{E:f.sigma.preserves.geom.indep}, $u_{0}, \ldots, u_{q}$ are geometrically independent. Hence, by \eqref{E:simp.map.commutes.with.bary}, $f_{\sigma}^{-1}(\zeta) = \langle u_{1}, \ldots, u_{q} \rangle$ is a $q$-dimensional simplex lying in $\sigma$. Obviously, the faces of $f_{\sigma}^{-1}(\zeta)$ have the form $\langle u_{i_{1}}, \ldots, u_{i_{k}} \rangle$ and $f_{\sigma} \bigl( \langle u_{i_{1}}, \ldots, u_{i_{k}} \rangle \bigr)$ is a face of $\zeta$.

Moreover, by \eqref{E:criterion.for.int.simp} and \eqref{E:simp.map.commutes.with.bary}, it is clear that 
	\begin{equation}   \label{E:for.L'.f.invrs.Int.iff.f.Int}   
		\text{For } \zeta \in L',\text{ we have } x \in \text{Int} \, f_{\sigma}^{-1}(\zeta)  
			\text{ if and only if } f_{\sigma}(x) \in \text{Int} \, \zeta 
	\end{equation} 
and if $\omega$ is a face of $\zeta \in L'$, then $f_{\sigma}^{-1}(\omega)$ is a face of $f_{\sigma}^{-1}(\zeta)$. Let $\zeta_{1}, \zeta_{2} \subset \tau$ be distinct simplices in $L'$. Then, by (\ref{E:intersection.of.simps}'), $\zeta_{1}$ and $\zeta_{2}$ have disjoint interiors. It follows from \eqref{E:f.sigma.is.1-1} that $f_{\sigma}^{-1}(\zeta_{1})$ and $f_{\sigma}^{-1}(\zeta_{2})$ have disjoint interiors. So by (\ref{E:intersection.of.simps}') again, the collection 
	 \[
		 \hat{P}_{\sigma} := \bigl\{ f_{\sigma}^{-1}(\zeta) : \zeta \in L', \; \zeta 
		   \subset f(\sigma) \bigr\}
	 \]
is a finite simplicial complex, a subdivision of the complex consisting of $\sigma$ and all its faces. Since $\tau$ is the union of simplices in $L'$, we must have that $\sigma$ is the union of the simplices in $\hat{P}_{\sigma}$, since $\tau = f(\sigma)$. 

Let $\psi$ be a face of $\sigma$. Then $\omega := f(\psi)$ is a face of $\tau$. If $\zeta \in L'$ and $\zeta \subset \omega$ then obviously $f_{\sigma}^{-1}(\zeta) \in \hat{P}_{\psi}$, because $f_{\psi}$ is just the restriction of $f_{\sigma}$ to $\psi$. If $\zeta \in L'$ lies 
in $\tau = f(\sigma)$, but $\zeta \nsubseteq \omega$ then $f_{\sigma}^{-1}(\zeta) \notin \hat{P}_{\psi}$. This proves
	\begin{equation}  \label{E:upward.compatibility.of.P.hat}
		\text{If } \sigma \in P \text{ and } \psi \text{ is a face of } \sigma, \text{ then } 
			\hat{P}_{\psi} = \{ \rho \in \hat{P}_{\sigma} : \rho \subset \psi \}. 
	\end{equation}

Let  
	\begin{equation*}
		\hat{P} := \bigcup_{\sigma \in P} \hat{P}_{\sigma}.
	\end{equation*}

\emph{Claim:} If $L' = L$, then $\hat{P} = P$. Suppose $L' = L$. Let $\sigma \in P$. Then, since $f$ is simplicial, 
$\zeta := f(\sigma) \in L = L'$. Therefore $\sigma = f_{\sigma}^{-1}(\zeta) \in \hat{P}$. Conversely, with $L' = L$ let $\xi \in \hat{P}$. Then there exists $\zeta \in L$ and $\sigma \in P$ s.t.\ $\zeta \subset f(\sigma)$ and $\xi = f_{\sigma}^{-1}(\zeta)$. It is obvious that $\xi$ is a face of $\sigma$, but let us continue. By proposition \ref{P:P/G}\eqref{I:f(rho)=omega}, there exists $\rho \in P$ s.t.\ $\zeta = f(\rho)$. Thus, $f(\rho) \cap f(\sigma) \neq \varnothing$. Hence, by proposition \ref{P:P/G}\eqref{I:f(rho).f(tau).dsjnt.iff.rho.Gtau.dsjnt}, there exists $g \in G$ s.t.\ $g(\rho) \cap \sigma \neq \varnothing$. Hence, by proposition \ref{P:P/G}\eqref{I:f(rho).f(tau).dsjnt.iff.rho.Gtau.dsjnt} again, $f \bigl[ g(\rho) \cap \sigma \bigr] = f(\rho) \cap f(\sigma) = \zeta \cap f(\sigma)$. Therefore, $\xi = f_{\sigma}^{-1}(\zeta) = g(\rho) \cap \sigma \in P$. This proves the claim.

\emph{Claim:} $\hat{P}$ is a simplicial complex. Since each $\hat{P}_{\sigma}$ ($\sigma \in P$) is a simplicial complex, \eqref{E:complex.contains.all.faces} holds with $P = \hat{P}$. It suffices then to show that (\ref{E:intersection.of.simps}') holds with $P = \hat{P}$. Let $\rho_{i} \in \hat{P}$ ($i=1,2$) and suppose $(\text{Int} \, \rho_{1}) \cap (\text{Int} \, \rho_{2}) \ne \varnothing$. Let $i=1,2$. There exists $\sigma_{i} \in P$ s.t.\ $\rho_{i} \subset \sigma_{i}$. We may assume that $\sigma_{i}$ is the smallest simplex in $P$ containing $\rho_{i}$. By \eqref{E:upward.compatibility.of.P.hat}, we have $\rho_{i} \in \hat{P}_{\sigma_{i}}$. But $\hat{P}_{\sigma_{i}}$ is a subdivision of $\sigma_{i}$. Therefore, by \eqref{E:inclusion.of.subdiv.interiors}, $\text{Int} \, \rho_{i} \subset \text{Int} \, \sigma_{i}$. 
Hence, $(\text{Int} \, \sigma_{1}) \cap (\text{Int} \, \sigma_{2}) \ne \varnothing$. By (\ref{E:intersection.of.simps}') applied to $P$ we have $\sigma_{1} = \sigma_{2}$. Thus, $\rho_{1}, \rho_{2} \in \hat{P}_{\sigma_{1}}$. But $\hat{P}_{\sigma_{1}}$ is a simplicial complex. Thus, by (\ref{E:intersection.of.simps}') applied to $P = \hat{P}_{\sigma_{1}}$, we have $\rho_{1} = \rho_{2}$. This completes the proof that $\hat{P}$ is a simplicial complex. It is immediate that $\hat{P}$ is a subdivision of $P$.

We have the following facts from proposition \ref{P:P/G}. (Recall that properties \eqref{I:G.simplicial.on.sdP} -- \eqref{I:f(rho)=omega} in proposition \ref{P:P/G} hold with $P$ in place of $P''$.)
	\begin{equation}  \label{E:f.surjective}
		f \bigl( |P| \bigr) = |L|. \quad 
		  \text{ (part \ref{I:P/G.is.a.complex} of proposition \ref{P:P/G}) }
	\end{equation}
	\begin{equation}  \label{E:fx.=.fy.iff.Gx.=.Gy}
		\text{If } x,y \in |P| \text{ then } f(x) = f(y) \text{ if and only if } Gy = Gx . 
		   \quad \text{ (part \ref{I:P/G.is.a.complex}) }
	\end{equation}
	\begin{multline}  \label{E:frho.=.ftau.iff.rho.=.gtau}
	\text{Let } \rho, \tau \in P. \text{ Then  } f(\rho) = f(\tau) \\
	   \text{ if and only if  there exists }
	  g \in G \text{ s.t.\  } \rho = g(\tau). 
	                   \quad \text{ (part \ref{I:f(rho).f(tau).dsjnt.iff.rho.Gtau.dsjnt}) }
	\end{multline}
	\begin{multline}  \label{E:ftau.cuts.frho.iff.Grho.cuts.tau}
	\text{If } \rho, \tau \in P, \text{ and } f(\tau) \cap f(\rho) \neq \varnothing \\
	 \text{ then for some $h \in G$, we have $h(\rho) \cap \tau \neq \varnothing$.} 
	           \quad \text{ (part \ref{I:f(rho).f(tau).dsjnt.iff.rho.Gtau.dsjnt}) }
	\end{multline}
We show that \eqref{E:f.surjective}, \eqref{E:fx.=.fy.iff.Gx.=.Gy}, \eqref{E:frho.=.ftau.iff.rho.=.gtau}, and \eqref{E:ftau.cuts.frho.iff.Grho.cuts.tau} all hold with $(\hat{P}, L')$ in place of $(P,L)$. We also show that 
	\begin{equation}  \label{E:part.of.prop.hold.for.P.hat}
		\text{Parts \eqref{I:G.simplicial.on.sdP}, \eqref{I:f.vs.geom.indep}, and \eqref{I:f(rho)=omega} 
		        of proposition \ref{P:P/G} hold with $(\hat{P}, L')$ in place of $(P'',L)$.}
	\end{equation}
That \eqref{E:f.surjective} and \eqref{E:fx.=.fy.iff.Gx.=.Gy} hold for $(\hat{P}, L')$ is immediate from the fact that they hold for $P$.

First we show that part \ref{I:f(rho)=omega} of proposition \ref{P:P/G} holds for $(\hat{P}, L')$. Let $\omega \in L'$. Then there exists $\zeta \in L$ s.t.\ $\omega \subset \zeta$. By part \ref{I:f(rho)=omega} of proposition \ref{P:P/G}, there exists $\sigma \in P$ s.t.\ $f(\sigma) = \zeta$. Hence, $\xi := f_{\sigma}^{-1}(\omega) \in \hat{P}$ and $f(\xi) = \omega$ as desired.

We prove that proposition \ref{P:P/G}(\ref{I:G.simplicial.on.sdP}) holds for $P'' = \hat{P}$. Proposition \ref{P:P/G}(\ref{I:G.simplicial.on.sdP}) applies to $P''=P$ by assumption. In particular, $G$ is a group of homeomorphisms of $|P| = |\hat{P}|$ onto itself. Suppose $v_{0}, \ldots, v_{p}$ span a simplex, $\rho \in \hat{P}$. Then $v_{0}, \ldots, v_{p}$ all lie in a simplex, $\sigma \in P$ and, by \eqref{E:f.sigma.preserves.geom.indep}, there exists $\psi \in L'$ s.t.\ $f(v_{0}), \ldots, f(v_{p})$ span $\psi$. Let $g \in G$. Then $f \bigl[ g(v_{i}) \bigr] = f(v_{i})$ ($i = 0, \ldots, p$) since \eqref{E:fx.=.fy.iff.Gx.=.Gy} applies to $P$. Thus, $f \bigl[ g(v_{0}) \bigr], \ldots, f \bigl[ g(v_{p}) \bigr]$ span $\psi$. Moreover, $g(v_{0}), \ldots, g(v_{p}) \in g(\sigma) \in P$. Hence, since $g$ is a simplicial homeomorphism on $P$, by lemma \ref{L:geom.indep.f(v).means.geom.indep.f(x)} we have that, $g(v_{0}), \ldots, g(v_{p})$ span a simplex $\xi$ lying in $g(\sigma) \in P$. Moreover, $f \bigl[ g(v_{0}) \bigr], \ldots, f \bigl[ g(v_{p}) \bigr]$ spans $\psi$. Thus, $\xi \in \hat{P}$ in $\hat{P}$. Similarly for $g^{-1}$. Thus, by definition \ref{D:simplicial.homeom}, $g$ is a simplicial homeomorphism of $\hat{P}$ onto itself. Therefore, proposition \ref{P:P/G}(\ref{I:G.simplicial.on.sdP}) holds for $\hat{P}$. I.e.,
	\begin{equation}  \label{E:G.simp.homeom.Phat}
		G \text{ is a group of simplicial homeomorphisms of } \hat{P} \text{ onto itself.}
	\end{equation}
(This will help prove the last sentence of the proposition.)

By definition of $\hat{P}$, As shown in the last paragraph, if $v_{0}, \ldots, v_{p}$ span a simplex in $\hat{P}$ then $f(v_{0}), \ldots, f(v_{p})$ span a simplex in $L'$. In particular, $f(v_{0}), \ldots, f(v_{p})$ are geometrically independent. Similarly, by \eqref{E:for.L'.f.invrs.Int.iff.f.Int}, if $\sigma \in \hat{P}$,
 then $f( \text{Int} \, \sigma) =  \text{Int} \, f(\sigma)$. So proposition \ref{P:P/G}(\ref{I:f.vs.geom.indep}) holds for $\hat{P}$.

We \emph{claim:} 
	\begin{multline}  \label{E:x.in.sigma.f.sigma.=.tau}
		\text{If } x \in |\hat{P}|, \tau \in L', \text{ and } f(x) \in \tau, \text{ then there exists } 
		      \sigma \in \hat{P} \text{ s.t.\ } f(\sigma) = \tau \text{ and } x \in \sigma. \\
		      \text{If } f(x) \in \text{Int} \, \tau \text{ then we may assume } 
		        x \in \text{Int} \, \sigma.
	\end{multline}
We prove this as follows. There exists $\omega \in L$ s.t.\ $\tau \subset \omega$. By proposition \ref{P:P/G}(\ref{I:f(rho)=omega}), there exists $\xi \in P$ s.t.\  $f(\xi) = \omega$. Hence, there exists $\rho \in \hat{P}$ s.t.\ $f(\rho) = \tau$. By \eqref{E:x.in.exctly.1.simplex.intrr}, $\tau$ has a face $\tau'$ s.t.\ $f(x) \in \text{Int} \, \tau'$. ($\tau' = \tau$ is possible.) Let $\rho'$ be the face of $\rho$ s.t.\ $f(\rho') = \tau'$. ($\rho'$ exists by definition of $\hat{P}$.) Then, by proposition \ref{P:P/G}(\ref{I:f.vs.geom.indep}) (which we now know applies to $(\hat{P}, L')$), $f(\text{Int} \, \rho') = \text{Int} \, \tau'$. Thus, there exists $y \in \text{Int} \, \rho'$ s.t.\ $f(y) = f(x)$. Therefore, by \eqref{E:fx.=.fy.iff.Gx.=.Gy}, $x = g(y)$, for some $g \in G$. Hence, by \eqref{E:G.simp.homeom.Phat} and  \eqref{E:simp.homeo.is.homeom}, $x \in  \text{Int} \, \sigma'$, where $\sigma' := g(\rho')$. We have $f(\sigma') = f \bigl[ g(\rho') \bigr] = f(\rho') = \tau'$ by \eqref{E:fx.=.fy.iff.Gx.=.Gy}.  
Similarly, if $\sigma := g(\rho)$, then $f(\sigma) = \tau$ and $x \in \sigma' \subset \sigma$. This proves the first sentence in \eqref{E:x.in.sigma.f.sigma.=.tau}. If $f(x) \in \text{Int} \, \tau$, then $\tau' = \tau$, $\rho' = \rho$, and $\sigma' = \sigma$ so $x \in \text{Int} \, \sigma$. This proves claim \eqref{E:x.in.sigma.f.sigma.=.tau}. The same proof shows that \eqref{E:x.in.sigma.f.sigma.=.tau} holds with $P$ in place of $\hat{P}$ and $L$ in place of $L'$.

We prove that \eqref{E:frho.=.ftau.iff.rho.=.gtau} holds for $P =  \hat{P}$. By \eqref{E:fx.=.fy.iff.Gx.=.Gy}, if $\rho \in \hat{P}$ and $g \in G$, then $f[g(\rho)] = f(\rho)$. Conversely, let $\rho, \tau \in \hat{P}$ and suppose $f(\rho) = f(\tau)$. Write $\psi := f(\rho) = f(\tau)$. By definition, $\psi \in L'$ and, since $L'$ is a subdivision of $L$, there exists $\omega \in L$ s.t.\ $\psi \subset \omega$. We may assume $\omega$ is the smallest simplex in $L$ that contains $\psi$. Then, by \eqref{E:inclusion.of.subdiv.interiors}, $\text{Int} \, \psi \subset \text{Int} \, \omega$. Let $x \in \text{Int} \, \rho$, $y \in \text{Int} \, \tau$. 
Then, by \eqref{E:for.L'.f.invrs.Int.iff.f.Int}, $f(x), f(y) \in \text{Int} \, \psi \subset \text{Int} \, \omega$. Now \eqref{E:x.in.sigma.f.sigma.=.tau} holds 
for $(P,L)$, so we may pick $\sigma, \zeta \in P$ s.t.\ $x \in \text{Int} \, \sigma$, $y \in \text{Int} \, \zeta$, and 
	\begin{equation}   \label{E:f.sigma.=.omega.=.f.zeta}
		f(\sigma) = \omega = f(\zeta).
	\end{equation}

Since $\hat{P}$ is a subdivision of $P$, by \eqref{E:inclusion.of.subdiv.interiors}, there exists $\xi \in P$ 
s.t.\  $\text{Int} \, \rho \subset \text{Int} \,  \xi$. So $x \in \text{Int} \, \xi$. Therefore, by (\ref{E:intersection.of.simps}'), 
we have $\sigma = \xi$ so $\rho \subset \sigma$. Similarly, $\tau \subset \zeta$. By \eqref{E:frho.=.ftau.iff.rho.=.gtau}
(applied to $P$) and \eqref{E:f.sigma.=.omega.=.f.zeta}, there exists $g \in G$ s.t.\ $\sigma = g(\zeta)$. Thus, $g(\tau) \subset \sigma$. Now, by \eqref{E:fx.=.fy.iff.Gx.=.Gy} and \eqref{E:f.sigma.=.omega.=.f.zeta}, we have 
	\[
		f_{\sigma} \bigl[ g(\tau) \bigr] = f \bigl[ g(\tau) \bigr] = f(\tau) = f(\rho) 
		  = f_{\sigma}(\rho). 
	\]
But, by \eqref{E:f.sigma.is.1-1}, $f_{\sigma}$ is one-to-one. Therefore, we must have $g(\tau) = \rho$. Thus, \eqref{E:frho.=.ftau.iff.rho.=.gtau} holds for $\hat{P}$. 

Next, we prove that \eqref{E:ftau.cuts.frho.iff.Grho.cuts.tau} holds with $\hat{P}$ in place of $P$. Suppose $\rho, \tau \in \hat{P}$ and $f(\rho) \cap f(\tau) \neq \varnothing$. Then there exist $x \in \rho$, $y \in \tau$ s.t.\ $f(x) = f(y)$. Hence, by \eqref{E:fx.=.fy.iff.Gx.=.Gy}, there exists $h \in G$ s.t.\ $y = h(x)$. I.e., $\tau \cap h(\rho) \neq \varnothing$ and \eqref{E:ftau.cuts.frho.iff.Grho.cuts.tau} is proved for $\hat{P}$.

Having gotten these few preliminaries out of the way, we now prove the proposition. Define 
	\begin{equation}   \label{E:S'L.=.f.S'}
		\Ss'_{L'} := f( \Ss' ). 
	\end{equation}
Then, since $\Ss'$ is compact and nonempty, $\Ss'_{L'}$ is closed and nonempty and, by \eqref{E:fx.=.fy.iff.Gx.=.Gy}, for $g \in G$ and 
$x \in |P|$, $f \big[ g(x) \bigr] \in \Ss'_{L'}$ if and only if $f(x) \in \Ss'_{L'}$. Suppose $x \in |P|$ and $f(x) \in \Ss'_{L'}$. 
Then there exists $y \in \Ss'$ s.t.\ $f(x) = f(y)$. Therefore, by \eqref{E:fx.=.fy.iff.Gx.=.Gy}, 
$x \in Gy$. Since $\Ss'$ is $G$-invariant, 
we have $x \in \Ss'$. I.e.,
	\begin{equation}   \label{E:f.invrs.S'.L.=S'}
		f^{-1}(\Ss'_{L'}) = \Ss' .
	\end{equation}
Define 
	\begin{equation} \label{E:Phi.L'.defn}
		\Phi_{L'} : |L'| \setminus \Ss'_{L'} \to \F \text{ by } \Phi_{L'} \bigl[ f(x) \bigr] 
		               = \Phi(x) \; (x \in |\hat{P}| \setminus \Ss').
	\end{equation}  
\emph{Claim:} $\Phi_{L'}$ is well-defined and continuous on $|L'| \setminus \Ss'_{L'}$. To see this, let $z \in |L'| \setminus \Ss'_{L'}$ Then, by \eqref{E:f.surjective} there exists $x \in |P|$ s.t.\ $z = f(x)$. Moreover, by \eqref{E:f.invrs.S'.L.=S'}, $x \notin \Ss'$ so $\Phi(x)$ is defined. Suppose $x, y \in |P|$ and $f(x) = f(y) = z \notin \Ss'_{L'}$. Then, as before, we have $x, y \in |P| \setminus \Ss'$. By \eqref{E:fx.=.fy.iff.Gx.=.Gy}, there exists $g \in G$ s.t.\ $y = g(x)$. But $\Phi$ is $G$-invariant. Therefore, $\Phi(y) = \Phi(x)$. This proves that $\Phi_{L'}$ is well-defined. 

Suppose $\{ z_{n} \} \subset |L'| \setminus \Ss'_{L'}$ and $z_{n} \to z \in |L'| \setminus \Ss'_{L'}$. There exists $\{ x_{n} \} \subset |P| \setminus \Ss'$ s.t.\ $f(x_{n}) = z_{n}$ for every $n$. Now, by \eqref{E:|P|.compact}, $|P|$ is compact. 
Therefore, WLOG, $x_{n} \to x \in |P|$. Now, by \eqref{E:simp.maps.are.Lip}, $f$ is continuous. Hence, $f(x) = z$. 
Therefore, $x \notin \Ss'$, so $\Phi(x)$ is defined, and $\Phi_{L'}(z_{n}) = \Phi(x_{n}) \to \Phi(x) = \Phi_{L'}(z)$. I.e, $\Phi_{L'}$ is continuous on $|L'| \setminus \Ss'_{L'}$. This completes the proof of the claim. 
 
 Let 
	 \begin{equation}   \label{E:hat.Q.QL'.defns}
		\hat{Q} := \bigl\{ \sigma \in \hat{P} : \sigma \subset |Q| \bigr\} 
		  \text{ and } Q_{L'} := f(\hat{Q}).
	\end{equation}
Thus, $|\hat{Q}| = |Q|$ and $f(\Ss' \cap \hat{Q}) = \Ss'_{L'} \cap Q_{L'}$. By definition of $\hat{P}$, we have that $Q_{L'}$ is a subcomplex of $L'$. Moreover, by proposition \ref{P:P/G}\eqref{I:G.simplicial.on.sdP} applied to $\hat{P}$ (see \eqref{E:part.of.prop.hold.for.P.hat}) and the fact that $g (Q) = Q$ for every $g \in G$, we have that 
	\begin{equation}  \label{E:Q.hat.G.invar}
		G \hat{Q} := \bigl\{ g(\sigma) : g \in G, \, \sigma \in \hat{Q} \bigr\} = \hat{Q}. 
	\end{equation}
If $\omega \in L'$ define 
$f^{-1}\bigl( \{\omega\} \bigr) := \bigl\{ \sigma \in \hat{P} : f(\sigma) = \omega \bigr\} \subset \hat{P}$. Therefore, 
by \eqref{E:frho.=.ftau.iff.rho.=.gtau} (with $(\hat{P}, L')$ in place of $(P,L)$), and proposition \ref{P:P/G}(\ref{I:G.simplicial.on.sdP}), applied to $(\hat{P}, L')$ (see \eqref{E:part.of.prop.hold.for.P.hat}), we see that  
	\begin{equation}   \label{E:f.invrs.omega.in.hat.Q.iff.omega.in.QL'}
		\text{For } \omega \in L', \; f^{-1}\bigl( \{\omega\} \bigr) \subset \hat{Q} 
		          \text{ if and only if } f^{-1}\bigl( \{\omega\} \bigr) \cap \hat{Q} 
		            \neq \varnothing \text{ if and only if } \omega \in Q_{L'}. 
	\end{equation}

By \eqref{E:S'L.=.f.S'}, \eqref{E:simp.maps.are.Lip}, \eqref{E:Lip.magnification.of.Hm}, and the fact that, by assumption, $\dim \bigl( \Ss' \cap |Q| \bigr) \leq a$, we have $\dim \bigl( \Ss'_{L'} \cap |Q_{L'}| \bigr) \leq a$. Apply parts \ref{I:Phi.tilde.locally.Lip.if.Phi.is} through \ref{I:arb.fine.subdivision} of theorem \ref{T:polyApproxThm} to 
$P = L'$, $\Phi = \Phi_{L'}$, $Q = Q_{L'}$, $a$, and $\Ss' = \Ss'_{L'}$. Denote the resulting $\tilde{\Phi}$ and $\tilde{\Ss}$ by $\tilde{\Phi}_{L'}$ and $\tilde{\Ss}_{L'}$, resp. Now let 
	\begin{equation}   \label{E:S.tilde.=.f.invrs.S'.L}
		\tilde{\Ss} := f^{-1}(\tilde{\Ss}_{L'}).
	\end{equation}  
Thus, $\tilde{\Ss}$ is closed. Since $f \bigl( |\hat{P}| \bigr) = |L'|$ (by \eqref{E:f.surjective}) we have 
	\begin{equation}  \label{E:S.tilde.L'.=.f.S.tilde}
		\tilde{\Ss}_{L'} = f( \tilde{\Ss} ).
	\end{equation}
Note that, by \eqref{E:fx.=.fy.iff.Gx.=.Gy}, 
	\begin{equation}  \label{E:S.tilde.G.invar}
		g(\tilde{\Ss}) = \tilde{\Ss}, \quad \text{ for every } g \in G.
	\end{equation}
(This proves an assertion in the proposition.) Define  
	\begin{equation} \label{E:tilde.Phi.defn}
		\tilde{\Phi} := \tilde{\Phi}_{L'} \circ f
	\end{equation}
so $\tilde{\Phi}$ is defined and continuous on $|\hat{P}| \setminus \tilde{\Ss}$ and for every $x \in |\hat{P}| \setminus \tilde{\Ss}$, we have, by \eqref{E:S.tilde.G.invar} and \eqref{E:fx.=.fy.iff.Gx.=.Gy}, that $\tilde{\Phi} \bigl[ g(x) \bigr]$ is defined and constant in $g \in G$, proving another assertion of the proposition. 

By part \ref{I:S.tilde.subcomp} of theorem \ref{T:polyApproxThm} as applied to $(\Ss'_{L'}, \Phi_{L'})$, lemma \ref{L:invrs.is.subcmplx}, and \eqref{E:f.invrs.omega.in.hat.Q.iff.omega.in.QL'}, we see that $\tilde{\Ss}_{L'} \cap |Q_{L'}|$ is either empty or the underlying space of a subcomplex of $Q_{L'}$. By lemma \ref{L:invrs.is.subcmplx} and proposition \ref{P:P/G}(\ref{I:f.vs.geom.indep},\ref{I:f(rho)=omega}) (as applied to $(\hat{P}, L')$) if $\tau \in L'$ then $f^{-1}(\tau)$ is  the total space of a subcomplex of $\hat{P}$ whose dimension, i.e., dimension of its highest dimensional simplex, is $\dim \tau$. Hence, $\tilde{\Ss} \cap |\hat{Q}|$ is either empty or the underlying space of a subcomplex of the $\lfloor a \rfloor$-skeleton of $\hat{Q}$. That proves that \textbf{part \ref{I:S.tilde.subcomp}} of theorem \ref{T:polyApproxThm} holds 
for $\tilde{\Ss}$.

Suppose $\F$ is a metric space and $\Phi$ is locally Lipschitz off $\Ss'$. See below for the proof of the following.

\begin{lemma}  \label{L:Phi.Lip.Phi.L'.Lip}
Suppose $\Phi : |P| \setminus \Ss' \to \F$ is locally Lipschitz. 
Then $\Phi_{L'} : |L'| \setminus \Ss'_{L'} \to \F$ is also locally Lipschitz.
\end{lemma}

Thus, $\tilde{\Phi}_{L'}$ is locally Lipschitz off $\tilde{\Ss}_{L'}$ by theorem \ref{T:polyApproxThm}, part \ref{I:Phi.tilde.locally.Lip.if.Phi.is}. Hence, $\tilde{\Phi}$ is locally Lipschitz off $\tilde{\Ss}$, by \eqref{E:tilde.Phi.defn}, because, by \eqref{E:simp.maps.are.Lip}, $f$ is Lipschitz. Apply \eqref{E:comp.of.Lips.is.Lip}. Therefore, \textbf{part \ref{I:Phi.tilde.locally.Lip.if.Phi.is}} of theorem \ref{T:polyApproxThm} holds 
for $(\Phi, \Ss')$.

It follows from \eqref{E:simp.maps.are.Lip} and \eqref{E:f.sigma.is.1-1} that there is a constant $C \in (0, \infty)$ s.t.\ if $\sigma \in P$, then the restriction $f_{\sigma} := f \restriction_{\sigma}$ and -- see \eqref{E:f.sigma.is.1-1} -- $f_{\sigma}^{-1}$ are Lipschitz with Lipschitz constant $C$. Clearly, the same holds for the restriction $f \restriction_{\rho}$ with $\rho \in \hat{P}$ with the same $C$, no matter which subdivision $\hat{P}$ is. 

Let $\rho \in \hat{P}$. \emph{Claim:} 
	\begin{equation}  \label{E:G.Int.=.f.invrs.f.}
		G(\text{Int} \, \rho) = f^{-1} \bigl[ \text{Int} \, f(\rho) \bigr].
	\end{equation}
To see this suppose $\rho \subset \sigma \in P$ and let $\zeta = f(\rho) = f_{\sigma}(\rho) \in L'$. Then, by \eqref{E:for.L'.f.invrs.Int.iff.f.Int} and \eqref{E:fx.=.fy.iff.Gx.=.Gy},
	\begin{align*}
	x \in G(\text{Int} \, \rho) &\text{ if and only if there exists } g \in G 
	  \text{ s.t.\ } g^{-1}(x) \in \text{Int} \, \rho 
	             = \text{Int} \, f_{\sigma} ^{-1} (\zeta) \\
	   &\text{ if and only if there exists } g \in G \text{ s.t.\ } f_{\sigma} \circ g^{-1}(x) 
	     \in \text{Int} \, \zeta \\
	   &\text{ if and only if there exists } g \in G \text{ s.t.\ } g^{-1}(x) 
	     \in f_{\sigma}^{-1} (\text{Int} \, \zeta) \\
	   &\text{ if and only if } x \in f^{-1} (\text{Int} \, \zeta) =  f^{-1} 
	     \bigl[ \text{Int} \, f(\rho) \bigr]. \\
	\end{align*}
This proves the claim.

Let $\mcl{A} \subset |P|$ be $G$-invariant and let $\rho \in \hat{P}$. \emph{Claim:}
	\begin{multline}  \label{E:f.A.cap.Int.=.f(A).cap.Int.f}
		f \bigl( \mcl{A} \cap (\text{Int} \, \rho) \bigr) = f \bigl( \mcl{A} 
		  \cap G(\text{Int} \, \rho) \bigr)
			       = f(\mcl{A}) \cap \bigl(\text{Int} \, f(\rho) \bigr) \\
			          \text{ and } f^{-1} \Bigl[ f(\mcl{A}) \cap \bigl(\text{Int} 
			            \, f(\rho) \bigr) \Bigr]
			            = \mcl{A} \cap G(\text{Int} \, \rho).
	\end{multline}
By \eqref{E:fx.=.fy.iff.Gx.=.Gy}, \eqref{E:g.commutes.w/.set.ops}, and \eqref{E:G.Int.=.f.invrs.f.}, 
	\begin{multline*}
		f \bigl( \mcl{A} \cap (\text{Int} \, \rho) \bigr) = f \Bigl( G \bigl[ \mcl{A} 
		  \cap (\text{Int} \, \rho)\bigr] \Bigr)
		  = f \bigl[ (G \mcl{A}) \cap G(\text{Int} \, \rho) \bigr] \\
		    = f \bigl( \mcl{A} \cap G(\text{Int} \, \rho) \bigr) 
	              \subset f(\mcl{A}) \cap \bigl(\text{Int} \, f(\rho) \bigr).
	\end{multline*} 
Conversely, let $y \in f(\mcl{A}) \cap \bigl(\text{Int} \, f(\rho) \bigr)$. Then, by \eqref{E:for.L'.f.invrs.Int.iff.f.Int}, there exists $x \in \mcl{A}$ and $w \in \text{Int} \, \rho$ s.t.\ $f(x) = y = f(w)$. Therefore, by \eqref{E:fx.=.fy.iff.Gx.=.Gy} there exists $g \in G$ s.t.\ $g(x) = w$. But $\mcl{A}$ is $G$-invariant. So we may assume $x = w$. I.e., $x \in \mcl{A} \cap (\text{Int} \, \rho)$. The second part follows from \eqref{E:G.Int.=.f.invrs.f.}, \eqref{E:fx.=.fy.iff.Gx.=.Gy}, and the $G$-invariance of $\mcl{A}$. This concludes the proof of the claim \eqref{E:f.A.cap.Int.=.f(A).cap.Int.f}.

Let $\tau \in L'$. Then, by \eqref{E:part.of.prop.hold.for.P.hat} (see part \eqref{I:f.vs.geom.indep} of proposition \ref{P:P/G}) there exists $\rho \in \hat{P}$ s.t.\ $f(\text{Int} \, \rho) = \text{Int} \, \tau$. By \eqref{E:G.Int.=.f.invrs.f.}, $
G (\text{Int} \, \rho) = \text{Int} \, \tau$. By \eqref{E:simp.homeo.is.homeom}, we have 
$G (\text{Int} \, \rho) = \bigcup_{g \in G} \text{Int} \, g(\rho)$. 

Let $\tau' \in L'$ with $\tau' \neq \tau$ and let $\rho' \in \hat{P}$ satisfy 
$f(\text{Int} \, \rho') = \text{Int} \, \tau'$. 
Suppose $G (\text{Int} \, \rho) \cap G (\text{Int} \, \rho') \neq \varnothing$. Then there exist $g, h \in G$ s.t.\ $\text{Int} \, g(\rho) \cap \text{Int} \, h(\rho') \neq \varnothing$. Therefore, by (\ref{E:intersection.of.simps}'), we have $g(\rho) = h(\rho')$. Hence, $\tau' = f \bigl[ h(\rho') \bigr] = f \bigl[ h(\rho') \bigr] = \tau$, contradiction. 
Thus, $G (\text{Int} \, \rho) \cap G (\text{Int} \, \rho') = \varnothing$.

We conclude that there exist $\rho_{1}, \ldots, \rho_{\ell} \in \hat{P}$ s.t.\ if $\tau \in L'$ there exists exactly one $i = 1, \ldots, \ell$ s.t.\ 
$\text{Int} \, \tau = \text{Int} \, f(\rho_{i}) = f \bigl[ G (\text{Int} \, \rho_{i}) \bigr]$ and 
$|P|$ is the disjoint union of $G (\text{Int} \, \rho_{i})$ ($i=1, \ldots, \ell$).

By \eqref{E:S.tilde.=.f.invrs.S'.L}, \eqref{E:G.Int.=.f.invrs.f.}, and \eqref{E:x.in.sigma.f.sigma.=.tau},
	\begin{align}  \label{E:S.tilde.f.invrs}
		\tilde{\Ss} \cap (\text{Int} \, \rho) \subset \tilde{\Ss} \cap G(\text{Int} \, \rho) 
			&= f^{-1} \Bigl( \tilde{\Ss}_{L'} 
			  \cap \bigl(\text{Int} \, f(\rho) \bigr) \Bigr) \notag \\ 
			&= \bigcup_{\tau \in L'} f^{-1} \left( \tilde{\Ss}_{L'} 
			  \cap \bigl(\text{Int} \, f(\rho) \bigr)  \cap (\text{Int} \, \tau) \right) \\
			&= \bigcup_{i=1}^{\ell} \bigcup_{g \in G} 
			  \bigl( f \restriction_{g(\text{Int} \, \rho_{i})} \bigr)^{-1}
			   \left( \tilde{\Ss}_{L'} \cap \bigl(\text{Int} \, f(\rho) \bigr) 
			     \cap \bigl( \text{Int} \, f(\rho_{i}) \bigr) \right). \notag 
	\end{align}
Similarly, using \eqref{E:f.invrs.S'.L.=S'},
	\begin{equation}  \label{E:S.prime.f.invrs}
		 \Ss' \cap (\text{Int} \, \rho) \subset \Ss' \cap G(\text{Int} \, \rho) 
		  = \bigcup_{i=1}^{\ell} \bigcup_{g \in G} 
		    \bigl( f \restriction_{\text{Int} \, g (\rho_{i})} \bigr)^{-1}
			   \left( \Ss'_{L'}\cap \bigl(\text{Int} \, f(\rho) \bigr)  
			     \cap \bigl( \text{Int} \, f(\rho_{i}) \bigr) \right).
	\end{equation}

By \eqref{E:f.A.cap.Int.=.f(A).cap.Int.f}, \eqref{E:S.tilde.G.invar}, \eqref{E:S.tilde.L'.=.f.S.tilde}, the fact that $\Ss'$ is $G$-invariant by assumption, and \eqref{E:S'L.=.f.S'} we have
	\begin{align} \label{E:S.tilde.prime.f}
		f \bigl( \tilde{\Ss} \cap (\text{Int} \, \rho) \bigr)
			       &= \tilde{\Ss}_{L'} \cap \bigl(\text{Int} \, f(\rho) \bigr), \\
		f \bigl( \Ss' \cap (\text{Int} \, \rho) \bigr) 
		               &= \Ss'_{L'} \cap \bigl(\text{Int} \, f(\rho) \bigr). \notag
	\end{align}

The sets $\bigl( f \restriction_{g (\text{Int} \, \rho_{i}} \bigr)^{-1}
\left( \tilde{\Ss}_{L'} \cap \bigl(\text{Int} \, f(\rho) \bigr)  \cap \bigl( \text{Int} \, f(\rho_{i}) \bigr) \right)$ ($i=1,\ldots, \ell$, $g \in G$) are disjoint.
Recall that for every $\rho \in \hat{P}$ we have that the restriction $f \restriction_{\rho_{i}}$ and $\bigl( f \restriction_{\rho_{i}} \bigr)^{-1}$ are Lipschitz with Lipschitz constant $C < \infty$. Then, by \eqref{E:S.tilde.f.invrs} and \eqref{E:S.prime.f.invrs},  \eqref{E:S.tilde.prime.f}, and \eqref{E:Lip.magnification.of.Hm}, if $s \geq 0$ we have 
	\begin{align}   \label{E:Hm.SL'.tilde.S.bounds}
		\Hm^{s} \bigl( \tilde{\Ss} \cap G(\text{Int} \, \rho) \bigr) 
		&= \sum_{i=1}^{\ell} \sum_{g \in G} \Hm^{s}  
		  \Bigl[ \bigl( f \restriction_{g (\rho_{i})} \bigr)^{-1}
			   \left( \tilde{\Ss}_{L'} \cap \bigl(\text{Int} \, f(\rho) \bigr)  
			     \cap \bigl( \text{Int} \, f(\rho_{i}) \bigr) \right) \Bigr] \notag \\
		&\leq C^{s} \sum_{i=1}^{\ell} \sum_{g \in G} \Hm^{s}  \Bigl[ 
			   \tilde{\Ss}_{L'} \cap \bigl(\text{Int} \, f(\rho) \bigr)  
			     \cap \bigl( \text{Int} \, f(\rho_{i}) \bigr) \Bigr] \\
		&= C^{s} |G| \sum_{\tau \in L'} \Hm^{s}  \Bigl[ \tilde{\Ss}_{L'} 
		  \cap \bigl(\text{Int} \, f(\rho) \bigr) \cap ( \text{Int} \, \tau ) \Bigr] 
		     \notag \\
		&= C^{s} |G| \Hm^{s}  \Bigl[ \tilde{\Ss}_{L'} 
		  \cap \bigl(\text{Int} \, f(\rho) \bigr)  \Bigr]  \notag
	\end{align}
Similarly,
	\begin{equation}
		\Hm^{s} \bigl( \Ss' \cap G(\text{Int} \, \rho) \bigr) 
		  \leq C^{s} |G| \Hm^{s}  \Bigl[ \Ss'_{L'}\cap \bigl(\text{Int} \, f(\rho) \bigr)  \Bigr].
	\end{equation}
More directly, from \eqref{E:S.tilde.prime.f}, we have
	\begin{multline}  \label{E:S.cap.(f).Int.rho.ineqs}
		\Hm^{s} \Bigl( \tilde{\Ss}_{L'} \cap \bigl(\text{Int} \, f(\rho) \bigr) \Bigr)
		   \leq C^{s} \Hm^{s} \bigl( \tilde{\Ss} \cap (\text{Int} \, \rho) \bigr)  \\
		      \text{ and } \Hm^{s} \Bigl( \Ss'_{L'} \cap \bigl(\text{Int} \, f(\rho) \bigr) \Bigr)
	       	        \leq C^{s} \Hm^{s} \bigl( \Ss' \cap (\text{Int} \, \rho) \bigr). 
	\end{multline} 
	
Let $\rho_{1}, \ldots, \rho_{\ell} \in \hat{P}$ be as above. Then, from \eqref{E:Hm.SL'.tilde.S.bounds}, we get for every $s \geq 0$,
	\begin{align}  \label{E:Hs.S.tilde.G.decomp}
		\Hm^{s} (\tilde{\Ss}) &= \sum_{i=1}^{\ell} \Hm^{s} \bigl( \tilde{\Ss} 
		  \cap G(\text{Int} \, \rho_{i}) \bigr) \notag \\
		  &\leq |G| C^{s} \sum_{i=1}^{\ell} \Hm^{s} \Bigl[ \tilde{\Ss}_{L'} 
		    \cap \bigl(\text{Int} \, f(\rho_{i}) \bigr) \Bigr] \\
		  &= |G| C^{s} \sum_{\tau \in L'} \Hm^{s} \bigl( \tilde{\Ss}_{L'} 
		    \cap (\text{Int} \, \tau ) \bigr) \notag \\
		  &= |G| C^{s} \Hm^{s} (\tilde{\Ss}_{L'}). \notag
	\end{align}
Similarly, for every $s \geq 0$,
	\begin{equation}  \label{E:Hs.S.prime.G.decomp}
		\Hm^{s} (\tilde{\Ss}_{L'}) \leq C^{s} \Hm^{s} (\tilde{\Ss}), \; 
		  \Hm^{s} (\Ss') \leq  |G| C^{s} \Hm^{s} (\Ss'_{L'}), 
		   \text{ and }
		     \Hm^{s} (\Ss'_{L'}) \leq C^{s} \Hm^{s} (\Ss'). 
	\end{equation}
Therefore,
	\begin{equation} \label{E:whole.space.dim.eqs}
		\dim \tilde{\Ss} = \dim \tilde{\Ss}_{L'} \text{ and } \dim \Ss'_{L'} 
		  = \dim \Ss'.
	\end{equation}

But by \eqref{E:f.invrs.omega.in.hat.Q.iff.omega.in.QL'} and \eqref{E:frho.=.ftau.iff.rho.=.gtau}, we have that $|\hat{Q}|$ is a disjoint union of some $G (\text{Int} \, \rho_{i})$'s and for such $i$'s 
$f \bigl( G (\text{Int} \, \rho_{i}) \bigr)$ is a disjoint cover of $Q_{L'}$. Hence, 
	\begin{multline}  \label{E:Hs.S.prime.tilde.Q.G.decomp}
	       \Hm^{s} \bigl( \tilde{\Ss} \cap |\hat{Q}| \bigr) 
	         \leq |G| C^{s} \Hm^{s} \bigl( \tilde{\Ss}_{L'} \cap |Q_{L'}| \bigr), \;
		\Hm^{s} \bigl( \Ss' \cap |\hat{Q}| \bigr) 
		  \leq  |G| C^{s} \Hm^{s} \bigl( \Ss'_{L'} \cap |Q_{L'}| \bigr), \; \\
		  \Hm^{s} \bigl( \tilde{\Ss}_{L'} \cap |Q_{L'}| \bigr) 
		    \leq C^{s} \Hm^{s} \bigl( \tilde{\Ss} \cap |\hat{Q}| \bigr), \text{ and }
		    \Hm^{s} \bigl( \Ss'_{L'} \cap |Q_{L'}| \bigr) 
		      \leq C^{s} \Hm^{s} \bigl( \Ss' \cap |\hat{Q}| \bigr). 
	\end{multline}
Therefore, 
	\begin{equation} \label{E:Q.dim.eqs}
		\dim \bigl( \tilde{\Ss} \cap |\hat{Q}| \bigr) 
		  = \dim \bigl( \tilde{\Ss}_{L'} \cap |Q_{L'}| \bigr) 
		  \text{ and } \dim \bigl( \Ss'_{L'} \cap |Q_{L'}| \bigr) 
		    = \dim \bigl( \Ss' \cap |\hat{Q}| \bigr).
	\end{equation}
By assumption $\Ss'$, etc., satisfy the assumptions of theorem \ref{T:polyApproxThm}. 
In particular, $\dim \bigl( \Ss' \cap |\hat{Q}| \bigr) \leq a$. Hence, by \eqref{E:Q.dim.eqs}, 
we have $\dim \bigl( \Ss'_{L'} \cap |Q_{L'}| \bigr) \leq a$ so $\Ss'_{L'}$, etc. satisfy the assumptions of theorem \ref{T:polyApproxThm}.

Applying part \ref{I:dim.Stilde.no.bggr.thn.dim.S} of theorem \ref{T:polyApproxThm} to $(\Phi_{L'}, \Ss'_{L'}, Q_{L'}, L')$, we see that 
$\dim \tilde{\Ss}_{L'} \leq \dim \Ss'_{L'}$ and $\dim \bigl( \tilde{\Ss}_{L'} \cap |Q_{L'}| \bigr) \leq \dim \bigl( \Ss'_{L'} \cap |Q_{L'}| \bigr)$. Therefore, by \eqref{E:whole.space.dim.eqs} and \eqref{E:Q.dim.eqs}, we see that \textbf{part \ref{I:dim.Stilde.no.bggr.thn.dim.S}}
of theorem \ref{T:polyApproxThm} also holds for $(\Phi, \Ss', \hat{Q}, \hat{P})$.

Suppose $\tau \in \hat{P}$ has the following property. 
	\begin{equation}   \label{E:tau.stays.away.from.S'}
		\text{If } \rho \in \hat{Q} \text{ and } (\text{Int} \, \rho) \cap \Ss' \neq \varnothing 
		   \text{ then } \tau \cap \rho = \varnothing. 
	\end{equation} 
Suppose $\omega \in Q_{L'}$ and $(\text{Int} \, \omega) \cap \Ss'_{L'} \ne \varnothing$. By \eqref{E:hat.Q.QL'.defns} and \eqref{E:S.tilde.prime.f}, there exist $\rho \in \hat{Q}$ s.t.\ 
	\begin{equation}  \label{E:f.rho.hat.=.rho}
		f(\rho) = \omega \text{ and } (\text{Int} \, \rho ) \cap \Ss' \neq \varnothing.
	\end{equation}
Since $\Ss'$ and $\hat{Q}$ are $G$-invariant, by \eqref{E:simp.homeo.is.homeom}, we have
	\begin{equation}  \label{E:h.rho.hat.in.hat.Q}
		\text{For every } h \in G \text{ we have } f \bigl[ h(\rho) \bigr] = \omega, \;
		          h( \rho ) \in \hat{Q}, 
			\text{ and } \bigl(\text{Int} \, h(\rho) \bigr) \cap \Ss' \neq \varnothing.
	\end{equation}
Now, if $f( \rho ) \cap f(\tau) = \omega \cap f(\tau) \neq \varnothing$, then, by \eqref{E:ftau.cuts.frho.iff.Grho.cuts.tau}, there exist $h \in G$ s.t.\ $h(\rho) \cap \tau \neq \varnothing$. But, by \eqref{E:h.rho.hat.in.hat.Q}, this contradicts \eqref{E:tau.stays.away.from.S'}. Therefore, $f(\tau) \cap \omega = \varnothing$.
Hence, by part \ref{I:no.change.off.nbhd.of.S.in.Q} of theorem \ref{T:polyApproxThm}, 
	\begin{equation}  \label{E:L'.tilde.L'.agreement}
		\tilde{\Ss}_{L'} \cap f(\tau) = \Ss'_{L'} \cap f(\tau) \text{ and } 
		       \tilde{\Phi}_{L'}\text{ and } \Phi_{L'} 
		         \text{ agree on } f(\tau) \setminus \Ss'_{L'}.
	\end{equation}   
               
We prove
	\begin{equation}  \label{E:S.tilde.cap.tau.=.S'.cap.tau}
		\tilde{\Ss} \cap \tau = \Ss' \cap \tau.
	\end{equation}
Let $\Ss_{a}$ and $\Ss_{b}$ each be either $\Ss'$ or $\tilde{\Ss}$ but not the same as each other. (We introduce $\Ss_{a}$ and $\Ss_{b}$ in order to get a ``one size fits all'' proof.) Let $\Ss_{a,L'} := f(\Ss_{a})$ and $\Ss_{b,L'} := f(\Ss_{b})$. Then we know, by \eqref{E:S'L.=.f.S'}, \eqref{E:S.tilde.L'.=.f.S.tilde}, and \eqref{E:L'.tilde.L'.agreement}, that
$\Ss_{a,L'} \cap f(\tau) = \Ss_{b,L'} \cap f(\tau)$. We need to show that $\Ss_{a} \cap \tau = \Ss_{b} \cap \tau$. 
Suppose $x \in \Ss_{a} \cap \tau$. We want to show that $x \in \Ss_{b} \cap \tau$. We have that $x \in \Ss_{a} \cap \tau$ implies 
$f(x) \in f(\Ss_{a}) \cap f(\tau) = \Ss_{a,L'} \cap f(\tau) = \Ss_{b,L'} \cap f(\tau)$. Therefore, there exists $y \in \Ss_{b}$ s.t.\ $f(y) = f(x)$. By \eqref{E:fx.=.fy.iff.Gx.=.Gy}, $Gx = Gy$. I.e., there exists $g \in G$ s.t.\ $g(y) = x \in \tau$. Now, by assumption $\Ss'$ is $G$-invariant and, by \eqref{E:S.tilde.G.invar}, $\tilde{\Ss}$ is $G$-invariant. Hence, $x = g(y) \in g(\Ss_{b}) \cap \tau = \Ss_{b} \cap \tau$, as desired. This proves \eqref{E:S.tilde.cap.tau.=.S'.cap.tau}. 

If $w \in \tau \setminus \Ss'$, then, by \eqref{E:f.invrs.S'.L.=S'}, $f(w) \in f(\tau) \setminus \Ss'_{L'} = f(\tau) \setminus \tilde{\Ss}_{L'}$. Thus, by \eqref{E:Phi.L'.defn} and \eqref{E:tilde.Phi.defn}, if $w \in \tau \setminus \Ss'$, then
	\[
		\Phi(w) = \Phi_{L'}\bigl[ f(w) \bigr] = \tilde{\Phi}_{L'} \bigl[ f(w) \bigr] 
		  = \tilde{\Phi}(w).
	\] 
Therefore, \textbf{part \ref{I:no.change.off.nbhd.of.S.in.Q}} of theorem \ref{T:polyApproxThm}, applies 
to $(\tilde{\Ss}, \tilde{\Phi})$ as well.

Let $\rho \in \hat{P} \setminus \hat{Q}$. (But $\rho \cap |\hat{Q}| \neq \varnothing$ is possible.)  To begin with, we show that $f(\rho) \in L' \setminus Q_{L'}$. Suppose $f(\rho) \in Q_{L'}$ and let $x \in \text{Int} \, \rho$. By proposition \ref{P:P/G}\eqref{I:f.vs.geom.indep}, $f(x) \in \text{Int} \, f(\rho)$. 
Then, by \eqref{E:hat.Q.QL'.defns}, for some $\sigma \in \hat{Q}$, we have $f(x) \in f(\sigma)$. Thus, $\bigl( \text{Int} \, f(\rho) \bigr) \cap f(\sigma) \neq \varnothing$. 
Hence, by \eqref{E:Int.rho.cuts.sigma.then.rho.in.sigma}, $f(\rho)$ is a face of $f(\sigma)$. Therefore, by 
\eqref{E:hat.Q.QL'.defns}, there is a face, $\omega \in \hat{Q}$, of $\sigma$ s.t.\ $f(\omega) = f(\rho)$. Thus, by \eqref{E:frho.=.ftau.iff.rho.=.gtau}, $\rho = g(\omega)$ for some $g \in G$. But $\hat{Q}$ is $G$-invariant. It follows that $\rho = g(\omega) \in \hat{Q}$, contradiction. We conclude that $f(\rho) \notin Q_{L'}$. 

Let $s \geq 0$. By part \ref{I:no.change.off.Q} of theorem \ref{T:polyApproxThm}, we have 	
	    \begin{equation}  \label{E:S'.L'.0.f.S.tilde.L'.0}
		\text{If } \Hm^{s} \Bigl( \Ss'_{L'} \cap \bigl[ \text{Int} \, f(\rho) \bigr] \Bigr) = 0
			 \text{ then } \Hm^{s} \Bigl( \tilde{\Ss}_{L'} 
			   \cap \bigl[ \text{Int} \, f(\rho) \bigr] \Bigr) = 0.
	\end{equation}
Suppose $\Hm^{s} \bigl[ \Ss' \cap (\text{Int} \, \rho ) \bigr] = 0$. 
Then, by \eqref{E:S.cap.(f).Int.rho.ineqs},
$\Hm^{s} \Bigl( \Ss'_{L'} \cap \bigl[ \text{Int} \, f(\rho) \bigr] \Bigr) = 0$. Therefore, by \eqref{E:S'.L'.0.f.S.tilde.L'.0}, we have 
$\Hm^{s} \Bigl( \tilde{\Ss}_{L'} \cap \bigl[ \text{Int} \, f(\rho) \bigr] \Bigr) = 0$. Hence, by \eqref{E:Hm.SL'.tilde.S.bounds},
 $\Hm^{s} \bigl[ \tilde{\Ss} \cap (\text{Int} \, \rho ) \bigr] = 0$. Thus, \textbf{part \ref{I:no.change.off.Q}} of the theorem holds 
 for $(\Phi, \tilde{\Ss})$.

If $\tau \in \hat{Q}$ and 
$\Hm^{\lfloor a \rfloor} \bigl( \tilde{\Ss} \cap (\text{Int} \, \tau) \bigr) > 0$,  then, by \eqref{E:hat.Q.QL'.defns}, 
$f(\tau) \in Q_{L'}$ and, by \eqref{E:Hm.SL'.tilde.S.bounds}, 
$\Hm^{\lfloor a \rfloor} \Bigl( \tilde{\Ss}_{L'} \cap \bigl(\text{Int} \, f(\tau) \bigr) \Bigr) > 0$. Hence, by part \ref{I:Ha.tau.=0.if.Ha.all.nbrs.0} of theorem \ref{T:polyApproxThm},  
$f(\tau) \in L'$ is an $\lfloor a \rfloor$-simplex and 
$\Hm^{\lfloor a \rfloor} \bigl[ \Ss'_{L'} \cap ( \text{Int} \, \psi ) \bigr] > 0$ for some simplex 
$\psi$ of $Q_{L'}$ having $f(\tau)$ as a face. 
($\psi = f(\tau)$ is possible.) By proposition \ref{P:P/G}\eqref{I:f(rho)=omega}, $\tau$ is an $\lfloor a \rfloor$-simplex.
Let $x \in \text{Int} \, \tau$. Thus, $f(x) \in \psi$. By \eqref{E:x.in.sigma.f.sigma.=.tau}, there exists $\rho \in \hat{P}$ s.t.\ $x \in \rho$ 
and $f(\rho) = \psi$. Therefore, by \eqref{E:Int.rho.cuts.sigma.then.rho.in.sigma}, 
$\tau$ is a face of $\rho$ and 
$\Hm^{\lfloor a \rfloor} \Bigl[ \Ss'_{L'} \cap \bigl( \text{Int} \, f(\rho) \bigr) \Bigr] > 0$.
Therefore, by \eqref{E:S.cap.(f).Int.rho.ineqs}, 
$\Hm^{\lfloor a \rfloor} \bigl( \Ss' \cap (\text{Int} \, \rho) \bigr) > 0$. 
Thus, \textbf{part \ref{I:Ha.tau.=0.if.Ha.all.nbrs.0}} of the theorem holds 
for $(\Phi, \tilde{\Ss})$.
	
If $y \in \tilde{\Ss}$ then, by \eqref{E:S.tilde.L'.=.f.S.tilde}, $f(y) \in \tilde{\Ss}_{L'}$ so by part \ref{I:dist.tween.S.tilde.and.S} of theorem \ref{T:polyApproxThm}, there exists $\tau \in L'$ s.t.\ $f(y) \in \tau$ and $\tau \cap \Ss'_{L'} \neq \varnothing$. By \eqref{E:x.in.sigma.f.sigma.=.tau}, there exists $\sigma \in \hat{P}$ such that $y \in \sigma$ and $f(\sigma) = \tau$. Thus, by \eqref{E:f.invrs.S'.L.=S'}, $\sigma \cap \Ss' \neq \varnothing$. Similarly, if $y \in \Ss'$, by \eqref{E:S'L.=.f.S'}, $f(y) \in \Ss'_{L'}$. So by part \ref{I:dist.tween.S.tilde.and.S} of theorem \ref{T:polyApproxThm}, there exists $\tau \in L'$ s.t.\ $f(y) \in \tau$ 
and $\tau \cap \tilde{\Ss}_{L'} \neq \varnothing$. By \eqref{E:x.in.sigma.f.sigma.=.tau} again, there exists $\sigma \in \hat{P}$ such that $y \in \sigma$ and $f(\sigma) = \tau$. Thus, by \eqref{E:S.tilde.=.f.invrs.S'.L}, $\sigma \cap \tilde{\Ss} \neq \varnothing$. This proves that \textbf{part \ref{I:dist.tween.S.tilde.and.S}} of theorem \ref{T:polyApproxThm} 
applies to $(\Phi, \Ss')$.

Let $\sigma \in \hat{P}$. Then $f(\sigma) \in L'$ and, by part \ref{I:Phi.tilde.sigma.subset.Phi.sigma} of theorem \ref{T:polyApproxThm} and \eqref{E:S'L.=.f.S'}, we have the following.
	\begin{equation} \label{E:f.sigma.S.tilde.S.prime}
		\tilde{\Phi}_{L'} \bigl[ f(\sigma) \setminus \tilde{\Ss}_{L'} \bigr] 
			\subset \Phi_{L'} \bigl[ f(\sigma) \setminus \Ss'_{L'} \bigr]. 
	  \end{equation}
Let $\Ss = \Ss'$ or $\tilde{\Ss}$. Giving $\Ss_{L'}$ the obvious meaning, we have by \eqref{E:S.tilde.L'.=.f.S.tilde} and \eqref{E:S'L.=.f.S'} that $\Ss_{L'} = f(\Ss)$. Now, $f(\sigma) = f(\sigma \cap \Ss) \cup f(\sigma \setminus \Ss)$. If $y \in f(\sigma \cap \Ss)$ then $y \in \Ss_{L'}$. Thus, if $x \in f(\sigma) \setminus \Ss_{L'}$ then $x \in f(\sigma \setminus \Ss)$. I.e., 
	\begin{equation}   \label{E:f.sigma.SL'.f.S}
		  f(\sigma) \setminus \Ss_{L'} \subset f(\sigma \setminus \Ss), 
		     \quad (\Ss = \Ss' \text{ or } \tilde{\Ss}).
	\end{equation} 

By \eqref{E:S.tilde.L'.=.f.S.tilde}, $f(\sigma) \setminus \tilde{\Ss}_{L'} \subset f(\sigma \setminus \tilde{\Ss})$. By \eqref{E:S.tilde.=.f.invrs.S'.L}, we have $f(\sigma \setminus \tilde{\Ss}) \cap \tilde{\Ss}_{L'} = \varnothing$. Thus,
$f(\sigma) \setminus \tilde{\Ss}_{L'} \subset f(\sigma \setminus \tilde{\Ss}) \subset f(\sigma) \setminus \tilde{\Ss}_{L'}$. I.e.,
	\begin{equation} \label{E:f.S.=.f.sigma.less.S}
		f(\sigma \setminus \tilde{\Ss}) = f(\sigma) \setminus \tilde{\Ss}_{L'}.
	\end{equation} 

By \eqref{E:tilde.Phi.defn}, \eqref{E:f.S.=.f.sigma.less.S}, \eqref{E:f.sigma.SL'.f.S}, \eqref{E:f.sigma.S.tilde.S.prime}, and \eqref{E:Phi.L'.defn} we have,
	\begin{align*}
		\tilde{\Phi}( \sigma \setminus \tilde{\Ss}) 
		    &= \tilde{\Phi}_{L'} \circ f( \sigma \setminus \tilde{\Ss}) \\
		    &=  \tilde{\Phi}_{L'} \bigl[ f(\sigma) \setminus \tilde{\Ss}_{L'} \bigr]   \\
		    &\subset \Phi_{L'} \bigl[ f(\sigma) \setminus  \Ss'_{L'} \bigr]   \\
		    &\subset \Phi_{L'} \circ  f( \sigma \setminus \Ss') \\
		    &=  \Phi ( \sigma \setminus \Ss').                    
	\end{align*} 
Thus, \textbf{part \ref{I:Phi.tilde.sigma.subset.Phi.sigma}} of theorem \ref{T:polyApproxThm} holds 
for $(\Phi, \Ss')$. 

We now prove part \ref{I:bound.on.S.tilde.vol}. \eqref{E:Hs.S.prime.tilde.Q.G.decomp},  
\eqref{E:polyhedral.volume.magnification.factor} applied to $(\Phi_{L'}, \Ss'_{L'}, Q_{L'})$, and \eqref{E:Hs.S.prime.tilde.Q.G.decomp} again,   
imply that there exists $K < \infty$ s.t.\
	\begin{align}  \label{E:Ha.S'.Ha.Stilde.via.SL'}
		\Hm^{a} \bigl( \tilde{\Ss} \cap |\hat{Q}| \bigr)  
		  &\leq |G| \, C^{a} \, \Hm^{a} \bigl( \tilde{\Ss}_{L'} \cap |Q_{L'}| \bigr) 
		    \notag \\
		  &\leq K |G| \, C^{a} \, \Hm^{a} \bigl( \Ss'_{L'} \cap |Q_{L'}| \bigr) \\
		   &\leq K |G| \, C^{2a} \, \Hm^{a} \bigl( \Ss' \cap |\hat{Q}| \bigr). \notag
	\end{align}
Thus, \eqref{E:polyhedral.volume.magnification.factor} applies to $(\Phi,\Ss', \hat{Q})$ and \textbf{part \ref{I:bound.on.S.tilde.vol}} of theorem \ref{T:polyApproxThm} is proved.

We prove that \textbf{part \ref{I:arb.fine.subdivision}} of theorem \ref{T:polyApproxThm} applies to $(\Phi,\Ss', \hat{Q})$. 
Let $(\Phi_{L}, \Ss'_{L}, Q_{L})$ be the triple $(\Phi_{L'}, \Ss'_{L'}, Q_{L'})$ with $L' = L$. (Recall that $L'$ is an arbitrary subdivision of $L$. $L' = L$ is possible.)
Part \ref{I:arb.fine.subdivision} applies to $(L, \Phi_{L}, \Ss'_{L}, Q_{L})$. Thus, there exists $K < \infty$, depending only on $a$ and $L$ (hence, only on $a$ and $P$ and $G$) with the following property. If $\epsilon > 0$, there exists a partition, $L'$, of $L$, which only depends on $L$ and $\epsilon$, s.t.\ for every $\zeta \in L'$ we have 
	\begin{equation}  \label{E:diam.zeta.<.eps/C}
		diam(\zeta) < \epsilon/C, \text{ for every } \zeta \in L',
	\end{equation}
($C$ is the constant in \eqref{E:Hm.SL'.tilde.S.bounds}; it is a Lipschitz constant of $f_{\sigma}$ and 
$f_{\sigma}^{-1} = (f \restriction_{\sigma})^{-1}$ for every $\sigma \in \hat{P}$.) \emph{And} parts \ref{I:Phi.tilde.locally.Lip.if.Phi.is} through \ref{I:Phi.tilde.sigma.subset.Phi.sigma} above and \eqref{E:polyhedral.volume.magnification.factor} hold when $P$ is replaced by $L'$ and $Q$ is replaced by $Q_{L'}$ with a ``new'' $K$ independent of $\epsilon$. By \eqref{E:diam.zeta.<.eps/C} and choice of $C < \infty$, 
we have $diam \bigl[ f_{\sigma}^{-1}(\zeta) \bigr] < \epsilon$. To prove that \eqref{E:polyhedral.volume.magnification.factor} holds for 
$(\hat{P}, \Ss')$, just use \eqref{E:Ha.S'.Ha.Stilde.via.SL'} but with the new $K$. By \eqref{E:G.simp.homeom.Phat}, $G$ is a group of simplicial homeomorphisms on 
$P' := \hat{P}$.
      \end{proof}

  \begin{proof}[Proof of lemma \ref{L:Phi.Lip.Phi.L'.Lip}]
 Let $z \in |L'| \setminus \Ss'_{L'}$. By \eqref{E:x.in.exctly.1.simplex.intrr}, there exists a unique simplex $\tau_{z} \in L'$ s.t.\ 
$z \in \text{Int} \, \tau_{z}$. By \eqref{E:St.sigma.is.open}, the star, $\text{St} \, \tau_{z}$, of $\tau_{z}$ is open. If $\tau \in L'$ and $z \in \tau$, then, by \eqref{E:Int.rho.cuts.sigma.then.rho.in.sigma}, $\tau_{z} \subset \tau$. Thus, 
	\begin{equation} \label{E:Star.bar.for.z}
		\overline{\text{St}} \, \tau_{z} = \{ \tau \in L' : z \in \tau \}.
	\end{equation} 
By lemma \ref{L:invrs.is.subcmplx}, 
$K := \bigcup_{z \in \tau \in L'} f^{-1}(\tau) =  f^{-1}(\overline{\text{St}} \, \tau_{z})$ is a subcomplex of $\hat{P}$. By proposition \ref{P:P/G}(\ref{I:f(rho)=omega}), 
$f \bigl( |K| \bigr) = \overline{\text{St}} \, \tau_{z}$. 

Now, $\Ss'_{L'}$ and the lemma depend on $|L'|$ but not on $L'$ as a complex. Therefore, for purposes of proving this lemma, since $\Ss'_{L'}$ is closed, we may, by Munkres \cite[Theorem 15.4, p.\ 86]{jrM84}, assume that $L'$ is fine enough that 
$\overline{\text{St}} \, \tau_{z} \cap \Ss'_{L'} = \varnothing$.

Let $x \in f^{-1}(z)$. Then $x \in | \hat{P} | \setminus \Ss'$. By assumption, $x$ has a neighborhood $U \subset |P| \setminus \Ss'$ s.t.\ $\Phi$ is Lipschitz on $U$ with Lipschitz constant $C_{1} < \infty$, say. Since $\Phi$ is $G$-invariant, $\Phi$ is Lipschitz on every set $g(U)$ ($g \in G$) with the same Lipschitz constant and, by \eqref{E:fx.=.fy.iff.Gx.=.Gy}, all the sets $f \circ g(U)$ ($g \in G$) are equal. 

By lemma \ref{L:local.finiteness.and.compactness}(\ref{I:|P|.is.subspace}), $U \cap \sigma$ is open in $\sigma$ for every $\sigma \in \hat{P}$. Moreover, by \eqref{E:f.sigma.is.1-1} and Simmons \cite[Theorem E, p.\ 131]{gfS63}, we have $f(U \cap \sigma) = f_{\sigma}(U \cap \sigma)$ is open in $f(\sigma)$ for every $\sigma \in \hat{P}$. Therefore, by lemma \ref{L:local.finiteness.and.compactness}(\ref{I:|P|.is.subspace}) again we have that $f(U)$ is open in $|L'|$.

Let $V = f(U) \cap \text{St} \, \tau_{z}$, so $V \subset |L'| \setminus \Ss'_{L'}$ is an open neighborhood of $z$, and let $w_{1}, w_{2} \in V$. Then, by \eqref{E:Star.bar.for.z}, there exists $\tau_{i} \in L'$ s.t.\ $w_{i} \in \text{Int} \, \tau_{i}$ and $z \in \tau_{i}$ ($i=1,2$). In particular, $\tau_{1} \cap \tau_{2} \neq \varnothing$. 
By proposition \ref{P:P/G}(\ref{I:f(rho)=omega}), there exists $\sigma_{i} \in \hat{P}$ s.t. $f(\sigma_{i}) = \tau_{i}$ ($i=1,2$). 

Let $\zeta = \tau_{1} \cap \tau_{2}$, so $z \in \zeta$. Then $\sigma_{i}$ has a face $\xi_{i}$ s.t.\ $f(\xi_{i}) = \zeta$. Therefore, by \eqref{E:frho.=.ftau.iff.rho.=.gtau}, there exists $g \in G$ s.t.\ $\xi_{2} = g(\xi_{1})$. Replace $\sigma_{1}$ by $g(\sigma_{1})$. Then $f(\sigma_{1}) \cap f(\sigma_{2}) = \zeta$. 

By \eqref{E:f.sigma.is.1-1} and \eqref{E:f.sigma.preserves.geom.indep}, the map $f_{\sigma_{i}}^{-1}$ is simplicial. Hence, by \eqref{E:simp.maps.are.Lip}, $f_{\sigma_{i}}^{-1}$ is is Lipschitz, with Lipschitz constant $C_{2} < \infty$, say. Since $\hat{P}$ is a finite complex, we may assume that $C_{2}$ does not depend on $\sigma_{i}$ ($i=1,2$).

Recall $\tau_{1} \cap \tau_{2} \neq \varnothing$. But first suppose neither $\tau_{1}$ nor $\tau_{2}$ is a face of the other. Therefore, if $d$ is the metric on $\F$ then, by corollary \ref{C:reverse.triangle.ineq.in.simp.cmplxs}, there exist $\tilde{w}_{i} \in \tau_{1} \cap \tau_{2}$ (so $\Phi_{L'}(\tilde{w}_{i})$ is defined) ($i=1,2$) and $C_{3} < \infty$ depending only on $\hat{P}$, s.t.\, 
	\begin{equation*}
		| w_{1} -  \tilde{w}_{1} | + | \tilde{w}_{1} -  \tilde{w}_{2} | +  | \tilde{w}_{2} - w_{2} | 
		  \leq C_{3} | w_{1} -  w_{2} |.
	\end{equation*}
Now, $\tilde{w}_{2} \in \zeta$. Therefore, both $f_{\sigma_{1}}^{-1}(\tilde{w}_{2})$ and $f_{\sigma_{2}}^{-1}(\tilde{w}_{2})$ are defined. Hence, since $\Phi$ is $G$-invariant, by \eqref{E:fx.=.fy.iff.Gx.=.Gy}, we have 
$\Phi \circ f_{\sigma_{1}}^{-1}(\tilde{w}_{2}) =  \Phi \circ f_{\sigma_{2}}^{-1}(\tilde{w}_{2})$. We use that fact in the following.
	\begin{multline*}
	    d \bigl[ \Phi_{L'}(w_{1}),  \Phi_{L'}(w_{2}) \bigr] \\
		\begin{aligned}
			\quad {} &\leq d \bigl[ \Phi_{L'}(w_{1}),  \Phi_{L'}(\tilde{w}_{1}) \bigr] 
			+ d \bigl[ \Phi_{L'}(\tilde{w}_{1}),  \Phi_{L'}(\tilde{w}_{2}) \bigr]  
			+  d \bigl[ \Phi_{L'}(\tilde{w}_{2}),  \Phi_{L'}(w_{2}) \bigr]  \\
			&= d \bigl[ \Phi \circ f_{\sigma_{1}}^{-1}(w_{1}),  
			       \Phi \circ f_{\sigma_{1}}^{-1}(\tilde{w}_{1}) \bigr] 
			+ d \bigl[ \Phi \circ f_{\sigma_{1}}^{-1}(\tilde{w}_{1}),  
			      \Phi \circ f_{\sigma_{1}}^{-1}(\tilde{w}_{2}) \bigr]  \\
			&\qquad 
			    +  d \bigl[ \Phi \circ f_{\sigma_{2}}^{-1}(\tilde{w}_{2}),  
			      \Phi \circ f_{\sigma_{2}}^{-1}(w_{2}) \bigr]  \\
			&\leq C_{1} \Bigl( \bigl| f_{\sigma_{1}}^{-1}(w_{1}) 
			  -  f_{\sigma_{1}}^{-1}(\tilde{w}_{1}) \bigr| 
			+ \bigl| f_{\sigma_{1}}^{-1}(\tilde{w}_{1}) 
			  -  f_{\sigma_{1}}^{-1}(\tilde{w}_{2}) \bigr|  
			+  \bigl| f_{\sigma_{2}}^{-1}(\tilde{w}_{2}) 
			  - f_{\sigma_{2}}^{-1}(w_{2}) \bigr| \Bigr) \\
			&\leq C_{1} C_{2} \bigl( | w_{1} -  \tilde{w}_{1} | 
			+ | \tilde{w}_{1} -  \tilde{w}_{2} | +  | \tilde{w}_{2} - w_{2} | \bigr) \\
			&\leq C_{1} C_{2} C_{3} | w_{1} -  w_{2} |.
		\end{aligned}
	\end{multline*}
If $\tau_{1}$ is a face of $\tau_{2}$ or vice versa, the proof is similar but easier.
  \end{proof}

\chapter{Facts Concerning Least Absolute Deviation Linear Regression}  \label{Chptr:LAD.technicalities} (This appendix is referred to in section \ref{SS:LAD}.) In this appendix, we use ``$Y$'' rather than ``$x$'' to denote data sets. Recall that $k \geq 1$ and $\nvar = k+1$. The $i^{th}$ row of $Y$ is $(x_{i}, y_{i})$, where $x_{i}$ is $1 \times k$ and $y_{i} \in \RR$ ($i \in \NN_{n}$; recall \eqref{E:N.sub.n}.). Alternatively, write $Y = (X^{n \times k}, y^{n \times 1})$. Recall that $Y \in \Y$ is ``collinear'' if $x_{2} - x_{1}, \ldots, x_{n}- x_{1}$ do not span $\RR^{k}$ (by \eqref{E:n>nvar>k>0}, $n > \nvar$; definition \ref{D:collinearity}). Therefore, the set of collinear data sets is closed. Let 
	\begin{equation*}
		X_{1}^{n \times \nvar} := 
			\begin{pmatrix}
					\begin{matrix}
						1 & x_{1}
					\end{matrix} \\
				\vdots \\
					\begin{matrix}
						1 & x_{n}
					\end{matrix} 
			\end{pmatrix}.
	\end{equation*}
 
By \eqref{E:nothing.special.about.x1.in.collin}, 
	\begin{multline}   \label{E:can.sbtrct.any.row.for.collin}
		X \text{ is collinear if and only if for any } i = 1, \ldots n  \\
		     \text{ the vectors } x_{2} - x_{i}, \ldots, x_{n}- x_{i} 
		        \text{ do not span } \RR^{k}.
	\end{multline}

Recall that in $L^{1}$ or Least Absolute Deviation (LAD) regression one fits the plane $y = \beta_{0} + x\beta_{1}$ (where $x$ is a $1 \times k$ row vector), where $\beta = (\beta_{0}, \beta_{1}^{T})^{T}$ with $\beta_{0} \in \RR$ and $\beta_{1}$ is $k \times 1$, and $b = (\beta_{0}, \beta_{1}^{T})^{T}$ minimizes 
	\begin{equation}  \label{E:L1.defn}
		L^{1}(b,Y) := \sum_{i=1}^{n} | y_{i} - b_{0} - x_{i} b_{i} |,
	\end{equation}
where $b = (b_{0}, b_{1}^{T})^{T}$ with $b_{0} \in \RR$ and $b_{1}$ is $k \times 1$. In this case say that the $k$-plane $\{ (x, \beta_{0} + x \beta_{1}): x \in \RR^{k} \}$ is an ``LAD plane'' and $\beta$ an ``LAD solution'' for $Y$. Let  $\hat{B}(Y)$ denote the set of all $\nvar \times 1$ column vectors $\beta$ minimizing $L^{1}(\beta, Y)$. Thus, by lemma \ref{L:Y.collnr.iff.rank.X1.<.k+1}, if $Y$ is collinear, then $\hat{B}(Y)$ is unbounded, because in that case $L^{1}(b,Y)$ is unchanged if we add an arbitrary vector to $b$ that is orthogonal to $X_{1}$.
Let $\hat{B}_{1}(Y) := \bigl\{ \beta_{1}^{k \times 1} : \text{There exists } \beta_{0} \in \RR \text{ s.t.\ } (\beta_{0}, \beta_{1}^{T})^{T} \in \hat{B}(Y) \bigr\}$ be the projection 
of $\hat{B}(Y)$ onto the last $k$ coordinates. $\hat{B}_{1}(Y)$ is convex because 
$\hat{B}(Y)$ is. The following result is reminiscent of  Bloomfield and Steiger \cite[Theorem 1, p.\ 7]{pBwlS83}. 

 \begin{lemma}  \label{L:basic.LAD.soln.facts}
  Suppose $Y \in \Y$ is not collinear. Then:

  a) $\hat{B}(Y)$ is nonempty, compact, and convex. In fact, there exists a neighborhood $\mcl{V} \subset \Y$ of $Y$ and a compact set $C \subset \RR^{k+1}$ s.t.\ if $Y' \in \mcl{V}$ then $Y'$ is not collinear and $\hat{B}(Y') \subset C$.

  b) There exists a nonempty set $U = U(Y)$ of $1 \times k$ unit vectors $u$ with the following properties:  

  \quad i) Given $u \in U$, there exists a unique 
$\bigl( \beta_{0}, (\beta_{1})^{T} \bigr)^{T} \in \hat{B}(Y)$ 
s.t.\ $\gamma_{1} = \beta_{1}$ maximizes $u   \gamma_{1}$ among 
all $\gamma_{1} \in \hat{B}_{1}(Y)$. 

  \quad ii) $U$ contains a basis of $\RR^{k}$ and if $u \in U$, then $-u \in U$.

  \quad iii)  If $u \in U$ and $\beta$ is as in (i) then 
	\begin{multline}  \label{E:exact.fit.basis}
		\text{For some } i_{1} < \cdots < i_{\nvar} \in \NN_{n}, 
		  \text{ we have } y_{i_{j}} 
		          = \beta_{0} + x_{i_{j}} \beta_{1}, \text{ for } j = 1, \ldots, \nvar, \\
		  \text{ and } x_{i_{2}} - x_{i_{1}}, \ldots , x_{i_{\nvar}} - x_{i_{1}} 
		    \text{ is a basis for } \RR^{k}.
	\end{multline}
 \end{lemma}
 
 $x_{i_{2}} - x_{i_{1}}, \ldots , x_{i_{\nvar}} - x_{i_{1}}$ is a basis for $\RR^{k}$ if and only if the $\nvar \times \nvar$ matrix whose rows are 
 $(1, x_{i_{1}}), \ldots, (1,x_{i_{\nvar}})$, has rank $\nvar$. This follows from definition \ref{D:collinearity} and lemma \ref{L:Y.collnr.iff.rank.X1.<.k+1}. Call the LAD solutions 
 $\beta = \bigl( \beta_{0}, (\beta_{1})^{T} \bigr)^{T}$ as 
 in part (b)(i) ``extreme LAD solutions''. 
 
The vector $u$ mentioned in part (b,i) clearly determines a supporting plane to the extreme point $\beta$ 
of $\hat{B}_{1}(Y)$. 

  \begin{corly} \label{C:B1.is.convex.hull.of.E}
Let $E$ be the set of $\beta_{1} \in \hat{B}_{1}(Y)$ for which there is a $u \in U$ for which part (b,i) of the lemma holds for $(u, \beta_{1})$. Then $E$ is finite and $\hat{B}_{1}(Y)$ is the convex hull of $E$
  \end{corly}
  \begin{proof}[Proof of corollary]
Let $\tilde{E}$ be the convex hull of $E$. For each $i_{1} < \cdots < i_{\nvar} \in \NN_{n}$ there can be at most one pair $(\beta_{0}, \beta_{1})$ satisfying \eqref{E:exact.fit.basis}. The number of subsets of $\NN_{n}$ of cardinality $\nvar$ is $\binom{n}{k} < \infty$. Thus, $E$ is finite, so $\tilde{E}$ is compact. 

By part (a) of the lemma, $\tilde{E} \subset \hat{B}(Y)$. 
Let $b_{1} \in \hat{B}_{1}(Y) \setminus \tilde{E}$ 
By hyperplane separation (Rao \cite[(i), p.\ 51]{crR73.LinStatInf}, 
Rudin \cite[Theorem 3.19, p.\ 70]{wR66.funcAnal} or see Wikipedia) there exists a unit vector $v \in \RR^{k}$ 
s.t.\ $b_{1} \cdot v > \gamma_{1} \cdot v$ for every $\gamma_{1} \in \hat{B}_{1}(Y)$.
Since, by part (b,ii) of the lemma, $U$ contains a basis, 
there exist $a_{1}, \ldots, a_{\ell} \in \RR$ 
and pairs $(u^{1}, \beta_{1}^{1}), \ldots, (u^{\ell}, \beta_{1}^{\ell})$ with $u^{i} \in U$ 
and $\beta_{1}^{i} \in E$ ($i = 1, \ldots, \ell$) as in part (b,i) s.t.\ 
$v = a_{1} u^{1} + \cdots + a_{\ell} u^{\ell}$. But by (b,ii) again, $-u^{i} \in U$. Therefore, we may assume $a_{i} > 0$ ($i = 1, \ldots, \ell$). Let $A = a_{1} + \cdots + a_{\ell} > 0$. Let $c_{i} := a_{i}/A > 0$ and $w := v/A \neq 0$,  
Then 
   \begin{subequations} 
	\begin{gather*}  
		c_{1} + \cdots + c_{\ell} = 1,  \\
		w = c_{1} u^{1} + \cdots + c_{\ell} u^{\ell},  \\
		c_{1} \beta_{1}^{1} + \cdots + c_{\ell} \beta_{1}^{\ell} \in \hat{B}_{1}(Y), 
	\end{gather*}
   \end{subequations} 
and
    \begin{equation*}
      c_{1} \beta_{1}^{1} \cdot w + \cdots + c_{\ell} \beta_{1}^{\ell} \cdot w 
        < b_{1} \cdot w = c_{1} b_{1} \cdot u^{1} + \cdots + c_{\ell} b_{1} \cdot u^{\ell}
          \leq c_{1} \beta_{1}^{1} \cdot w + \cdots + c_{\ell} \beta_{1}^{\ell} \cdot w ,
    \end{equation*}
This contradiction proves the corollary. 
  \end{proof}

\begin{proof}[Proof of lemma \ref{L:basic.LAD.soln.facts}]  
(a)  Since $\gamma \mapsto L^{1}(\gamma, Y)$ is convex in $\gamma$, $\hat{B}(Y)$ is convex. Again by continuity of $L^{1}(\gamma, Y)$, $\hat{B}(Y)$ is closed. We show $\hat{B}(Y)$ is bounded and nonempty. 

Let $\gamma = (\gamma_{0}, \gamma_{1}^{T})^{T}$ be a $\nvar \times 1$ column vector, where $\gamma_{0} \in \RR$ and $\gamma_{1}$ is ${k \times 1}$, and let 
$\| X_{1} \gamma \|$ denote the sum of the absolute values of the entries of 
$(X_{1} \gamma)^{n \times 1}$. $| X_{1} \gamma |$, on the other hand, is the ordinary Euclidean norm of $X_{1} \gamma$. Then, by the triangle inequality,
	\begin{equation*}
		\sum_{i=1}^{n} | \gamma_{0} + x_{i} \gamma_{1} | 
		          = \| X_{1} \gamma \| \geq | X_{1} \gamma |.
	\end{equation*}
Now,
	\begin{equation*}
		| X_{1} \gamma | = \sqrt{\gamma^{T} X_{1}^{T} X_{1} \gamma } \geq \sqrt{\lambda_{\nvar}} |  \gamma |,
	\end{equation*}
where $\lambda_{\nvar} \geq 0$ is the smallest eigenvalue of $X_{1}^{T}X_{1}$. Since, by assumption, $Y$ is not collinear, by lemma \ref{L:Y.collnr.iff.rank.X1.<.k+1}, 
$\lambda_{\nvar} > 0$. Therefore,
	\begin{equation}  \label{E:L1.gamma.geq.root.lambda.-.y}
		L^{1}(\gamma, Y) \geq \| X_{1} \gamma \| - \sum_{i=1}^{n} |y_{i}| 
		       \geq \sqrt{ \lambda_{\nvar}} |\gamma| - \sum_{1}^{n} |y_{i} |. 
	\end{equation} 
Let $\alpha$ be some column $\nvar$-vector and choose $r \in (0, \infty)$ satisfying
	\begin{equation}  \label{E:r.lambda.-.y.>.L1}
		r \sqrt{ \lambda_{\nvar} } - \sum_{1}^{n} |y_{i}| > L^{1}( \alpha, Y ).
	\end{equation}
Let $C := \{ \gamma \in \RR^{k+1}: |\gamma| \leq r \}$, so $C$ is compact. Thus, by \eqref{E:L1.gamma.geq.root.lambda.-.y} and \eqref{E:r.lambda.-.y.>.L1}, it suffices 
to minimize $L^{1}(\gamma, Y)$ in $\gamma$ in the compact set $C$. 
Since $L^{1}(\gamma, Y)$ is continuous in $\gamma$, the minimum is achieved. Thus, $\hat{B}(Y)$ is nonempty 
and $\hat{B}(Y) \subset C$. 
In fact, $L^{1}(\alpha, Y')$ is continuous in $Y' = (X', y') \in \Y$ and, by lemma \ref{L:Eigen.cont.}, $\lambda_{\nvar}$ is continuous in $X_{1}$ so \eqref{E:r.lambda.-.y.>.L1}, 
with $Y'$ in place of $Y$, holds in a neighborhood, $\mcl{V}$, of $Y$. Thus, $\hat{B}(Y') \subset C$ holds for every $Y' \in \mcl{V}$. Moreover, since the set of collinear data sets is closed and $Y$ is non-collinear, we may assume that all data sets in $\mcl{V}$ are non-collinear. \emph{(a) is proved.}

Let $\mcl{J}  \subset 2^{\NN_{n}}$ denote the collection of subsets 
$J = \{i_{1}, \ldots, i_{m}\}$ of $\NN_{n}$ s.t.\ $x_{i_{2}} - x_{i_{1}}, \ldots, x_{i_{m}} - x_{i_{1}}$ do \emph{not} span $\RR^{k}$ ($m=2, \ldots, n$). (It is easy to see that the span 
of $x_{i_{2}} - x_{i_{1}}, \ldots , x_{i_{m}} - x_{i_{1}}$ is independent of which element of $J$ is labeled $i_{1}$.) Corresponding to $m = 0$, let $\varnothing \in \mcl{J}$. Every subset of $\NN_{n}$ having $k$ or fewer elements is in $\mcl{J}$, so $\mcl{J}$ is nonempty. Since $Y$ is not collinear, by definition \ref{D:collinearity}, $\{ 1, \ldots, n \} \notin \mcl{J}$. If $J  = \{i_{1}, \ldots, i_{m}\} \in \mcl{J}$, let $S_{J} \subset \RR^{k}$ be the subspace spanned by $x_{i_{2}} - x_{i_{1}}, \ldots , x_{i_{m}} - x_{i_{1}}$. In particular, $S_{\varnothing} = \{ 0 \} \subset \RR^{k}$. Let $S = \bigcup_{J \in \mcl{J}} S_{J}$. Then $S$ is closed, has $k$-dimensional Lebesgue measure 0, is scale invariant, and $-S = S$. Let $U =  S^{k-1} \setminus S$. ($S^{k-1} \subset \RR^{k}$ is the $(k-1)$-dimensional unit sphere.)  Then $U$ has property (b,ii), since $S^{k-1} \setminus S$ is nonempty and open in $S^{k-1}$. 

Let $u \in U$. (Think of $u$ as a $1 \times k$ row vector). Then $|u| =1 $. Let 
	\begin{equation}   \label{E:defn.of.a.sup}
		a := \sup \bigl\{ u \gamma_{1}:  \gamma_{1} \in \hat{B}_{1}(Y) \bigr\} .
	\end{equation}
Let $\beta = (\beta_{0}, \beta_{1}^{T})^{T} \in \hat{B}(Y)$ be a point at which $\gamma_{1} = \beta_{1} \in \hat{B}_{1}(Y)$ is a vector at which the maximum in \eqref{E:defn.of.a.sup} is achieved. (Since $\hat{B}(Y)$ is compact, $a < +\infty$, and $\beta$ exists.)  Whether or not $\beta_{1}$ is the only such vector, we now prove \eqref{E:exact.fit.basis} holds with this $u$ and $\beta$. 

Reordering the rows of $Y$ if necessary, let $\kappa = 0, 1, \ldots, n$ satisfy 
	\begin{equation}  \label{E:yi.i.leq.m}
		y_{i} = \beta_{0} + x_{i} \beta_{1} \text{ if and only if } i = 1, \ldots, \kappa.
	\end{equation}
Let $S \subset \RR^{k}$ be the space spanned 
by $x_{2} - x_{1}, \ldots, x_{\kappa} - x_{1}$ ($S = \{0\}$ if $\kappa = 0$ or 1) and \emph{assume} $S \neq \RR^{k}$. Then by definition of $U$, we have $u \notin S$. 
Let $\hat{u}$ be the orthogonal projection of $u$ onto $S$. 
Let $v^{k \times 1} := (u-\hat{u})^{T}$. Then $v \perp S$ and $u v > 0$. 
($u$ is a row vector so $uv$ makes sense. Similarly, $x_{i}$ is a row vector so 
$x_{i} v \in \RR$.) For $t \in  \RR$, 
let $\gamma_{0}(t)^{1 \times 1} = \beta_{0} - t x_{1} v$ and 
$\gamma_{1}(t)^{k \times 1} = \beta_{1} + t v$. 
Write, $\gamma(t) = (\gamma_{0}(t), \gamma_{1}(t)^{T})^{T}$ and let 
	\begin{equation*}
		f(t) := \sum_{i > \kappa} \bigl| y_{i} - \gamma_{0}(t) 
		  - x_{i} \gamma_{1}(t) \bigr|.
	\end{equation*}
(Notice that the sum is over $i > \kappa$.)

By \eqref{E:yi.i.leq.m}, 
	\begin{equation*}
		\gamma_{0}(t) + x_{i} \gamma_{1}(t) 
		      = \beta_{0} + x_{i} \beta_{1} + t(x_{i} - x_{1})v = y_{i},
		        \quad i \leq \kappa .
	\end{equation*}
since $v \perp S$ and $x_{i} - x_{1} \in S$ if $i \leq \kappa$. 
Therefore, $L^{1} \bigl( \gamma(t), Y \bigr) = f(t)$, $t \in \RR$. (See \eqref{E:L1.defn}.) Thus, $f(0) = L^{1}(\beta, Y)$. By definition of $\kappa$ again, if $i > \kappa$,
$y_{i} - \gamma_{0}(0) - x_{i} \gamma_{1}(0) = y_{i} - \beta_{0} - x_{i} \beta_{1} \neq 0$. Therefore, there exist $\epsilon_{\kappa+1}, \ldots, \epsilon_{n} = \pm 1$ 
s.t.\ for some $\delta > 0$,
	\begin{equation*}
		f(t) = \sum_{i>\kappa} \epsilon_{i} \bigl[ y_{i} - \gamma_{0}(t) 
		  - x_{i} \gamma_{1}(t) \bigr], \quad |t| < \delta.
	\end{equation*}
Therefore, the derivative, $\tfrac{d}{ds} L \bigl(\gamma(s), Y  \bigr) \restriction_{s=t} = f'(t)$, exists. Now, $\gamma(0) = \beta \in \hat{B}(Y)$ so $L \bigl(\eta), Y  \bigr)$ is minimized by $\eta = \beta$. Therefore, we have $f '(0) = 0$. But $\gamma_{0}$ and $\gamma_{1}$ are linear functions. Thus, in $(-\delta, \delta)$ we have $f' = 0$. I.e., $\gamma(t) \in \hat{B}(Y)$ for every $t \in (-\delta, \delta)$. However, if $t \in (0, \delta)$,
	\begin{equation*}
		u \gamma_{1}(t) = u \beta_{1} + t u v >  u \beta_{i} = a. 
	\end{equation*}
But this contradicts the choice of $\beta$. Therefore, our assumption that 
$S \neq \RR^{k}$ must be false. I.e., $S = \RR^{k}$. This means $\kappa-1 \geq k$, 
i.e., $\kappa \geq \nvar$. Hence, \eqref{E:exact.fit.basis} holds.

Let $\beta$ be as before, \emph{viz.}, a point of $\hat{B}(Y)$ at which the maximum in \eqref{E:defn.of.a.sup} is achieved. We show that $\beta$ is unique. Suppose 
$\beta' = (\beta_{0}', \beta_{1}'^{T})^{T} \in \hat{B}(Y) \setminus \{ \beta \}$ and 
$u \beta_{1}' = u \beta_{1} = a$. For $r \in \RR$ let
$\beta(r) := \beta(r) = \bigl( \beta_{0}(r), \beta_{1}(r)^{T} \bigr)^{T} := r \beta + (1 - r) \beta'$. For every $r \in \RR$, $u [r \beta_{1} + (1 - r) \beta_{1}'] = a$. 
Since $\hat{B}(Y)$ is convex, as $r$ ranges from 0 to 1, $\beta(r) \in \hat{B}(Y)$ and takes on infinitely many values. 

By \eqref{E:exact.fit.basis}, for every $r \in [0, 1]$, there exist distinct
$i_{1} = i_{1}(r), \ldots, i_{\nvar} = i_{\nvar}(r) \in \NN_{n}$, s.t.\ 
$y_{i_{j}} = \beta_{0}(r) + x_{i_{j}} \beta_{1}(r)$, for $j = 1, \ldots, \nvar$, 
and $x_{i_{2}} - x_{i_{1}}, \ldots , x_{i_{\nvar}} - x_{i_{1}}$ is a basis for $\RR^{k}$. 
Let $\mbf{i} := \{ i_{1}, \ldots, i_{\nvar} \}$ and let
	 \begin{equation*}
	       y_{\mbf{i}} := (y_{i_{1}}, \ldots, y_{i_{\nvar}})^{T} \text{ and } X_{1,\mbf{i}} :=
		\begin{pmatrix}
				\begin{matrix}
					1 & x_{i_{1}} \\
					1 & x_{i_{2}}
				\end{matrix} \\
			\vdots \\
				\begin{matrix}
					1 & x_{i_{\nvar}}
				\end{matrix}
		\end{pmatrix}
		       .
	\end{equation*}
Then $y_{\mbf{i}} = X_{1,\mbf{i}} \, \beta(r)$. But, 
$x_{i_{2}} - x_{i_{1}}, \ldots , x_{i_{\nvar}} - x_{i_{1}}$ spans $\RR^{k}$. Hence, by lemma \ref{L:Y.collnr.iff.rank.X1.<.k+1} (with $n = \nvar$), the matrix $X_{1,\mbf{i}}$ is of full rank. That means, for any choice of $\{ i_{1}, \ldots, i_{\nvar} \}$ there is only one 
$\nvar \times 1$ matrix $\gamma = (\gamma_{0}, \gamma_{1}^{T})^{T}$ s.t.\ 
$y_{\mbf{i}} = X_{1,\mbf{i}} \, \gamma$, i.e.\ s.t.\ 
$y_{i_{j}} = \gamma_{0} + x_{i_{j}} \gamma_{1}$ for $j = 1, \ldots, \nvar$. Since there are only finitely many sequences $1 \leq i_{1} < \ldots < i_{\nvar} \leq n$, it follows that there can be at most finitely many $r \in [0, 1]$ s.t.\ $u \beta_{1}(r) = a$. This contradiction proves (i), uniqueness of $\beta$.
\end{proof}

Recall that $\Y_{LAD}'$ denotes the set of all $Y \in \Y$ s.t.\ $Y$ is not collinear and $\hat{B}(Y)$ contains exactly one point. 

\begin{remark}  \label{R:basis.for.Phi.LAD}
Let $Y \in \Y_{LAD}'$, so $Y$ is not collinear and $\hat{B}(Y)$ contains just one point, $(\beta_{0}, \beta_{1})$ (subsection \ref{SS:LAD}). By lemma \ref{L:basic.LAD.soln.facts}(b,(iii)), there exists $i_{1}, \ldots, i_{\nvar} \in \NN_{n}$ s.t.\ \eqref{E:exact.fit.basis} holds. \emph{Claim:} $( x_{i_{j}} - x_{i_{1}}, y_{i_{j}} - y_{i_{1}} )$ ($j = 2, \ldots, \nvar$) is a basis for the element, $\xi$, of $G(k,\nvar)$ parallel to the LAD plane of $Y$. By \eqref{E:exact.fit.basis} and definition of ``LAD plane'', $(x_{i_{j}}, y_{i_{j}} )$ lies on the LAD plane of $Y$ ($j = 2, \ldots, \nvar$). Therefore, $( x_{i_{j}} - x_{i_{1}}, y_{i_{j}} - y_{i_{1}} ) \in \xi$. But, again by \eqref{E:exact.fit.basis}, $x_{i_{2}} - x_{i_{1}}, \ldots, x_{i_{\nvar}} - x_{i_{1}}$ are linearly independent. Hence, $( x_{i_{j}} - x_{i_{1}}, y_{i_{j}} - y_{i_{1}} )$ is a basis of $\xi$.
\end{remark}

As remarked above, if $Y$ is collinear then the set $\hat{B}(Y)$ is unbounded. Therefore, if $\hat{B}(Y)$ contains only one point then $Y$ is automatically not collinear. We have the following. 
 
	\begin{prop}   \label{P:Y'.LAD.is.dense}
		$\Y_{LAD}'$ is dense in $\Y$.
	\end{prop}
Recall \eqref{E:1n.col.vec.defn}. Notice that if $Y \in \Y_{LAD}'$ then $a Y + c (1^{n}, 0^{n \times k}) \in \Y_{LAD}'$ for every $a, c \in \RR$ with $a \neq 0$. Thus, by proposition \ref{P:Y'.LAD.is.dense}, $\Y_{LAD}' \cap \D$ is dense in $\D$ defined by \eqref{E:plane.fitting.D.mu.defn}. (See \eqref{E:shft.invar.of.regrssn}.)

\begin{proof} By \eqref{E:noncollinear.dense.inY}, it suffices to show 
that $\Y_{LAD}'$ is dense in the set of all non-collinear data sets in $\Y$. Let $Y \in \Y$ be a non-collinear data set. If $Y \in \Pf^{k}$, then clearly $\Phi_{LAD}(Y) = \Delta(Y)$ is unique. I.e., $Y \in \Y'_{LAD}$. 

Suppose $Y$ is not collinear and not in $\Pf^{k}$ either and let $U$ be as in lemma \ref{L:basic.LAD.soln.facts}. Let $u \in U$ be arbitrary but fixed and 
let $\beta = \beta(u)\in \hat{B}(Y)$ be as in lemma \ref{L:basic.LAD.soln.facts}(b,i). Write $\beta = (\beta_{0}, \beta_{1}^{T})^{T}$. For a column 
$\nvar$-vector $\gamma = (\gamma_{0}, \gamma_{1}^{T})^{T}$, let 
	\begin{equation}  \label{E:eps.gamma.defn}
		\epsilon_{i}(\gamma) := sign(y_{i} - \gamma_{0} - x_{i} \gamma_{1}), 
		         \quad i \in \NN_{n}. 
	\end{equation} 
(See \eqref{E:sign.function}.) Suppose $\gamma = (\gamma_{0}, \gamma_{1}^{T})^{T} \neq \beta$ and define 
	\begin{equation}  \label{E:gamma.t.defn}
	   \gamma(t) = \bigl( \gamma_{0}(t), \gamma_{1}(t)^{T} \bigr)^{T} 
		  = \beta + t(\gamma - \beta), \quad t \in \RR. 
	\end{equation}
Thus, $\gamma(0) = \beta$ and so does not depend on $\gamma$. Suppose 
	\begin{equation*}
		u \gamma_{1} \geq a 
		   :=  \sup \bigl\{u \alpha_{1}: 
		       \alpha = (\alpha_{0}, \alpha_{1}^{T})^{T} \in \hat{B}(Y) \bigr\} 
		         = u\beta_{1}.
	\end{equation*}  
(Such a $\gamma_{1}$ exists, e.g., $\gamma_{1} = 2 \beta_{1}$.) If $t > 0$ then 
$\gamma(t) \ne \beta$, but $u \gamma_{1}(t) = u \beta_{1} + t (u \gamma_{1} -u \beta_{1}) \geq a$, so by lemma \ref{L:basic.LAD.soln.facts}(b,i) $\gamma(t) \notin \hat{B}(Y)$ and therefore 
	\begin{equation}  \label{E:L1.gamma.>.L1.beta}
		L^{1} \bigl( \gamma(t), Y \bigr) > L^{1}(\beta, Y), \text{ if } t > 0.
	\end{equation}

\emph{Claim:}  For some $t_{0} = t_{0}(\gamma) > 0$, $\epsilon_{i}(\gamma,t) :=  \epsilon_{i}(\gamma(t))$ ($i \in \NN_{n}$) are all constant in $t \in [0, t_{0}]$. Notice that 
	\begin{equation}  \label{E:resid.at.t}
		r_{i}(t) :=  y_{i} - \gamma_{0}(t) - x_{i} \gamma_{1}(t) 
		   =  y_{i} - \beta_{0} - x_{i} \beta_{1} 
		         - t \bigl[ (\gamma_{0} - \beta_{0}) + x_{i}(\gamma_{1} - \beta_{1}) \bigr].
	\end{equation}
If $(\gamma_{0} - \beta_{0}) + x_{i}(\gamma_{1} - \beta_{1}) \neq 0$,  $r_{i}(t)$ has exactly one 0. Call it $s$. Thus $y_{i} - \beta_{0} - x_{i} \beta_{1}$ and 
$s \bigl[ (\gamma_{0} - \beta_{0}) + x_{i}(\gamma_{1} - \beta_{1}) \bigr]$ have the same sign.
If $s \leq 0$, let $t_{0i}(\gamma) = +\infty$. Otherwise, let $t_{0i}(\gamma) = s  > 0$. In this case   
	\begin{equation}  \label{E:t0i.expression}
		0 < t_{0i}(\gamma) = \frac{y_{1} - \beta_{0} - x_{i} \beta_{1}}
	                  {(\gamma_{0} - \beta_{0}) + x_{i}(\gamma_{1} - \beta_{1}) }
	           = \frac{|y_{1} - \beta_{0} - x_{i} \beta_{1}|}
	                  {\bigl| (\gamma_{0} - \beta_{0}) + x_{i}(\gamma_{1} - \beta_{1}) \bigr|}
	              < \infty.
	\end{equation}
If $(\gamma_{0} - \beta_{0}) + x_{i}(\gamma_{1} - \beta_{1}) = 0$, and $y_{i} - \beta_{0} - x_{i} \beta_{1} = 0$, then, by \eqref{E:resid.at.t}, $r_{i}(t)$ is identically 0. In this case, again let $t_{0i}(\gamma) = +\infty$. If $(\gamma_{0} - \beta_{0}) + x_{i}(\gamma_{1} - \beta_{1}) = 0$, but $y_{i} - \beta_{0} - x_{i} \beta_{1} \neq 0$, then $r_{i}(t)$ has no 0's. Again, let $t_{0i}(\gamma) = +\infty$. Notice that for each $i \in \NN_{n}$, we have $t_{0i} > 0$ and 
$sign \bigl[ y_{i} - \gamma_{0}(t) - x_{i} \gamma_{1}(t) \bigr] = sign \, r_{i}(t)$ is constant 
in $t \in \bigl( 0, t_{0i}(\gamma) \bigr]$.

Finally, let 
	\begin{equation}   \label{E:t0(gamma).defn}
		t_{0}(\gamma) = 1 - \exp \bigl\{ -\min_{i}t_{0i}(\gamma) \bigr\}.
	\end{equation}  
Then $0 < t_{0}(\gamma) <  \min_{i}t_{0i}(\gamma)$ and $t_{0}(\gamma)$ is finite. Thus, there exists $t_{0} = t_{0}(\gamma) > 0$ s.t.\ 
$\epsilon_{i} :=  \epsilon_{i} \bigl( \gamma(t) \bigr)$ ($i \in \NN_{n}$) are all constant 
in $t \in (0, t_{0}]$ (i.e., excluding $t = 0$). But since $\gamma(t)$ is continuous in $t$, if $0 \leq t \leq t_{0}$ (i.e., including $t = 0$) 
$\epsilon_{i} \bigl( \gamma(t) \bigr) [y_{i} - \gamma_{0}(0) - x_{i} \gamma_{1}(0)] \geq 0$ 
($i \in \NN_{n}$). Define $\epsilon_{i}( \gamma, 0 ) := \epsilon \bigl[ \gamma(0) \bigr] :=  \epsilon_{i}$. This proves the claim.

Thus, by \eqref{E:gamma.t.defn}, 
	 \begin{multline} \label{E:f.gamma.defn}
	  f_{\gamma}(t) :=  L^{1} \bigl( \gamma(t), Y \bigr)
	    = \sum_{i=1}^{n} \epsilon_{i} 
	       \bigl( y_{i} - \gamma_{0}(t) - x_{i} \gamma_{1}(t) \bigr) \\
	         = \sum_{i=1}^{n} \epsilon_{i} 
	           \bigl( y_{i} - \beta_{0} - t(\gamma_{0}-\beta_{0}) 
	             - x_{i} \beta_{i} - t x_{i} (\gamma_{i}-\beta_{i}) \bigr)  .
		     \quad 0 \leq t \leq t_{0} .
	\end{multline}
Now, given $\gamma$, we have that $\epsilon_{1} \bigl[ \gamma(t) \bigr], \ldots, \epsilon_{n} \bigl[ \gamma(t) \bigr]$ are all constant for $0 \leq t < t_{0}(\gamma)$. Therefore, if $\gamma \in C$, then, by \eqref{E:ft.>.f0.f'.gamma.poz} and \eqref{E:f.gamma.defn},
	\begin{equation}  \label{E:exprssn.for.f'.gamma}
		0 < f'_{\gamma}(t) = f'_{\gamma}(0) 
			= - \sum_{i=1}^{n} \epsilon_{i} (\gamma, 0) 
				\bigl[ (\gamma_{0} - \beta_{0}) + x_{i} (\gamma_{1} - \beta_{1}) \bigr],  
					\quad 0 < t < t_{0}(\gamma).
	\end{equation}
Since $\gamma(t)$ is linear in $t$, the derivative $f_{\gamma}'$ is constant on $[0, t_{0})$. ($f_{\gamma}'(0)$ is the right hand derivative.)  For $t \in (0, t_{0})$ we have, by \eqref{E:L1.gamma.>.L1.beta},
	\begin{equation}  \label{E:ft.>.f0.f'.gamma.poz}
		f_{\gamma}(t) > f_{\gamma}(0), \text{ for } t \in (0, t_{0}). \text{ So }
			f_{\gamma}'(t) > 0, \text{ for }   t \in \bigl[ 0, t_{0}(\gamma) \bigr).
	\end{equation}

Define
	\begin{equation}  \label{E:rho.gamma.quo.defn}
		\rho_{\gamma} := \rho(\gamma) :=  \frac{f_{\gamma}'(0)}{u \gamma_{1} - a}.
	\end{equation}
(Define $\rho(\gamma) = +\infty$ if $u\gamma_{1} = a$.) Thus, 
	\begin{equation}   \label{E:rho.gamma.positive}
		\text{If } \gamma 
		   = (\gamma_{0}, \gamma_{1}^{T})^{T} \in \RR^{\nvar}, 
	              \gamma \neq \beta, \text{ and } u \gamma_{1} \geq a, 
	                  \text{ then } \rho(\gamma) > 0.
	\end{equation}

Let 
	\begin{equation}  \label{E:rho.inf.defn}
		\rho :=  \inf \bigl\{ \rho(\gamma): \gamma 
		         = (\gamma_{0}, \gamma_{1}^{T})^{T} \in \RR^{\nvar}, 
		             \gamma \neq \beta, \text{ and } u \gamma_{1} \geq a \bigr\}.
	\end{equation}  

\emph{Claim:}  
    \begin{equation}  \label{E:rho.>.0}
      \rho > 0 .
    \end{equation}
If  $\gamma = (\gamma_{0}, \gamma_{1}^{T})^{T} \ne \beta$ and $u \gamma_{1} \geq a$, let  $\tilde{\gamma}= \beta+|\gamma - \beta|^{-1}(\gamma-\beta)$. 
Then $u \tilde{\gamma}_{1} \geq a$. 
$\tilde{\gamma} \neq \beta$ and, by \eqref{E:f.gamma.defn}, 
	 \begin{align*} 
	  f_{\tilde{\gamma}}(t) &:=  L^{1} \bigl( \tilde{\gamma}(t), Y \bigr) \\
	         &= \sum_{i=1}^{n} \epsilon_{i} 
	           \bigl[ y_{i} - \beta_{0} - t(\tilde{\gamma}_{0}-\beta_{0}) 
	             - x_{i} \beta_{i} - t x_{i} (\tilde{\gamma}_{i}-\beta_{i}) \bigr]  .
		     \quad 0 \leq t \leq t_{0} \\
	        &= \sum_{i=1}^{n} \epsilon_{i} 
	           \bigl[ y_{i} - \beta_{0} - t |\gamma - \beta|^{-1}(\gamma_{0}-\beta_{0}) 
	             - x_{i} \beta_{i} - t x_{i} |\gamma - \beta|^{-1}(\gamma_{i}-\beta_{i}) \bigr] ,  \\
		     & \qquad \qquad 0 \leq t \leq t_{0} .
	\end{align*}

Therefore, by \eqref{E:exprssn.for.f'.gamma}, 
$f_{\tilde{\gamma}}'(0) = |\gamma - \beta|^{-1} f_{\gamma}'(0)$. Moreover, because 
$u \beta_{1} = a$, we have 
$u \tilde{\gamma}_{1} - a = |\gamma - \beta|^{-1} (u \gamma_{1} - a)$. 
Thus, by \eqref{E:rho.gamma.quo.defn}, $\rho(\tilde{\gamma}) = \rho(\gamma)$. Let
	\begin{equation}  \label{E:C.set.defn}
		C = \bigl\{ \gamma = (\gamma_{0}, \gamma_{1}^{T})^{T} \in \RR_{\nvar}: 
		   |\gamma - \beta| = 1 \text{ and } u \gamma_{1} \geq a \bigr\}.
	\end{equation}
So $C$ is compact. Therefore, we have
	\begin{equation*}
		\rho = \inf \bigl\{ \rho(\gamma): \gamma \in C \bigr\}.
	\end{equation*}

By \eqref{E:rho.gamma.positive}, in order to prove the claim \eqref{E:rho.>.0}, it suffices to show that $\rho( \cdot )$ is bounded away from 0 on $C$. While $C$ is compact, the function $\rho(\gamma)$ is not lower semi-continuous in $\gamma$ (Ash \cite[A6.1, p.\ 388]{rbA72}) so some work will be needed to show $\rho > 0$.

Since, by \eqref{E:gamma.t.defn}, $\gamma(t) \to \beta$ as $t \downarrow 0$, if $y_{i} - \beta_{0} - x_{i} \beta_{1} \neq 0$, then $ \epsilon_{i} (\gamma, 0) $ does not depend 
on $\gamma$. However, by lemma \ref{L:basic.LAD.soln.facts}(b, iii), there are at least 
$\nvar$ values of $i$ for which $y_{i} - \beta_{0} - x_{i} \beta_{1} =  0$. For those $i$'s, 
$ \epsilon_{i} (\gamma, 0) $ does depend on $\gamma$. 

Now, for each $i$ we have that $\bigl| (\gamma_{0} - \beta_{0}) + x_{i}(\gamma_{1} - \beta_{1}) \bigr|$ achieves a finite maximum value on $\gamma \in C$. Call that maximum value, $M_{i} < \infty$. Then, by \eqref{E:t0i.expression}, we have
	\begin{multline*}
		\text{For } \gamma \in C \text{ and } i \in \NN_{n}, \text{ we have } \\
		       t_{0i}(\gamma) 
		         \geq M_{i}^{-1} |y_{i} - \beta_{0} - x_{i} \beta_{1}| > 0 \text{ or }
			t_{0i}(\gamma) = + \infty.
	\end{multline*}
I.e., by \eqref{E:t0(gamma).defn}, there exists $\tilde{t}_{0} > 0$, s.t.\ 
$t_{0}(\gamma) \geq \tilde{t}_{0} > 0$ for every $\gamma \in C$. 

By compactness of $C$, continuity of $\gamma \mapsto f_{\gamma}(\tilde{t}_{0})$, and \eqref{E:ft.>.f0.f'.gamma.poz}, there exists $\eta > 0$, s.t.\
	\begin{equation*}
		f_{\gamma}(\tilde{t}_{0}) - f_{\gamma}(0) > \eta > 0, \quad \gamma \in C.
	\end{equation*}  
Since $0 < \tilde{t}_{0} \leq t_{0}$, we have that, for every $\gamma \in C$, $f_{\gamma}'(t)$ is constant in $t \in [0, \tilde{t}_{0}]$. Therefore,
	\begin{equation}   \label{E:lwr.bnd.f.gamma}
		f_{\gamma}'(0) = f_{\gamma}'(\tilde{t}_{0}) 
		      = \bigl[ f_{\gamma}(\tilde{t}_{0}) - f_{\gamma}(0) \bigr] / \tilde{t}_{0}
		      > \eta/\tilde{t}_{0}, \quad \gamma \in C.
	\end{equation}
On the other hand, by \eqref{E:C.set.defn}, 
	\begin{equation*}
		0 \leq u \gamma_{1} - a \leq |u| |\gamma_{1}| + |a| \leq |u| \bigl( |\beta| + 1 \bigr) 
		  + |a| < \infty, \quad \gamma \in C.
	\end{equation*}
Therefore, by \eqref{E:lwr.bnd.f.gamma},
	\begin{equation*}
		\rho(\gamma) \geq \frac{\eta}{\tilde{t}_{0} \Bigl[ \bigl( |\beta| + 1 \bigr) 
		 + |a| \Bigr]} > 0, \gamma \in C.
	\end{equation*}
This proves the claim \eqref{E:rho.>.0} that $\rho > 0$.

We need to find $Y' \in \Y_{LAD}'$ arbitrarily close to $Y$. Let $u \in U(Y)$, suppose 
 $\gamma = \beta \in \hat{B}(Y)$ maximizes $u \gamma_{1}$. Let $a := u \beta_{1}$. By lemma \ref{L:basic.LAD.soln.facts}(b,iii), WLOG (Without Loss Of Generality) for some 
$m = \nvar, \ldots, n$, 
	\begin{equation}   \label{E:exact.fit.for.i.leq.m}
		y_{i} = \beta_{0} + x_{i} \beta_{i} \text{ if and only if } i = 1, \ldots, m. 
	\end{equation}
Moreover, $x_{2} - x_{1}, \ldots, x_{m} - x_{1}$ spans $\RR^{k}$. If $m = n$, then $Y \in \Pf^{k}$, contrary to assumption. So $m < n$. Let $\epsilon_{i} = \epsilon_{i}(\beta) = \pm 1$ satisfy $\epsilon_{i} (y_{i} - \beta_{0} - x_{i} \beta_{i}) \geq 0$, $i \in \NN_{n}$. Let $\delta > 0$ be small. Define $x_{i}' = x_{i}$, $y_{i}' = y_{i}$ for $i \neq m+1$ and let 
	\begin{equation*}
		x_{m+1}' = x_{m+1} + \delta \epsilon_{m+1} u 
		   \text{ and } y_{m+1}' = y_{m+1} + \delta \epsilon_{m+1} a.
	\end{equation*}  
Let $Y' \in \Y$ be the matrix whose $i^{th}$ row is $(x_{i}', y_{i}')$. By making $\delta$ sufficiently small, we can make $Y'$ arbitrarily close to $Y$. We will show that for $\delta > 0$ sufficiently small, $Y' \in \Y_{LAD}'$. Since $Y$ is not collinear, if $\delta$ is sufficiently small, $Y'$ is not collinear. By \eqref{E:Y.LAD'.defn}, it remains to show that $\hat{B}(Y')$ is a singleton. We will show, in fact, that 
	\begin{equation}  \label{E:B.Y'.=beta}
		\text{For } \delta > 0 \text{ sufficiently small } \hat{B}(Y') = \{\beta \}.
	\end{equation}  
Since $\gamma \mapsto L^{1}(\gamma, Y')$ is convex, it suffices to show that $L^{1}(\gamma, Y')$ is uniquely minimized by $\gamma = \beta$ for $\gamma$ in a neighborhood of $\beta$.  

By definition of $m$, $y_{m+1} - \beta_{0} - x_{m+1} \beta_{1} \neq 0$. Hence, if $\delta > 0$ is sufficiently small,  $y_{m+1}' - \beta_{0} - x_{m+1}' \beta_{1}$ is not 0 either and has the same sign as $y_{m+1} - \beta_{0} - x_{m+1} \beta_{1}$. Let $\gamma^{\nvar \times 1} = (\gamma_{0}, \gamma_{1}^{T})^{T}$. Recall the definitions \eqref{E:eps.gamma.defn} and \eqref{E:gamma.t.defn}. By making $\gamma$ closer to $\beta$ if necessary, we may assume $\beta$ that $\epsilon_{i} = \epsilon_{i}[\gamma(t)] = \epsilon_{i} \; (\equiv \epsilon_{i}(\beta))$ and so is constant in $t \in [0,1]$ $(i \in \NN_{n})$. (Use the fact that, by \eqref{E:exact.fit.for.i.leq.m}, $\epsilon_{i}, \ldots, \epsilon_{m} = \pm 1$ can be arbitrary. Redefine them if necessary.)  By making $\gamma$ even closer to $\beta$ if necessary, we have $\epsilon_{m+1}(y_{m+1}' - \gamma_{0} - x_{m+1}' \gamma_{1}) > 0$. Then, 
	\begin{align}  \label{E:L1.Y'}
		L^{1}(\gamma, Y') &= \sum_{i=1}^{n} \epsilon_{i} \,
		                         (y_{i}' - \gamma_{0} - x_{i}' \gamma_{1}) \notag \\
			 &= \sum_{i=1}^{n} \epsilon_{i} \, (y_{i} - \gamma_{0} - x_{i} \gamma_{1}) 
			       + \delta a - \delta u \gamma_{1}  \\
			 &= L^{1}(\gamma, Y) + \delta a - \delta u \gamma_{1} \notag \\
			 &\geq L^{1}(\beta, Y) + \delta a - \delta u \gamma_{1},\notag 
	\end{align}
since $\beta \in \hat{B}(Y)$. In particular, since $u \beta_{1} = a$, we have 
	\begin{equation} \label{E:L.beta.Y'.=L.beta.Y}
		L^{1}(\beta, Y') = L^{1}(\beta, Y).
	\end{equation}

By \eqref{E:L1.Y'} and \eqref{E:L.beta.Y'.=L.beta.Y}, if $u \gamma_{1} < a$, $\gamma \notin \hat{B}(Y')$. Next, suppose $\gamma \neq \beta$, but $u\gamma_{1} = a$. By lemma \ref{L:basic.LAD.soln.facts}(b,i), we have
that $\gamma \notin \hat{B}(Y)$. Thus, by \eqref{E:L1.Y'} and \eqref{E:L.beta.Y'.=L.beta.Y},
	\begin{equation*}
		L^{1}(\gamma, Y') = L^{1}(\gamma,Y) > L^{1}(\beta, Y) = L^{1}(\beta, Y').
	\end{equation*}
I.e., $\gamma \notin \hat{B}(Y')$.

Suppose $u \gamma_{1} > a$. WLOG $\delta \in (0, \rho)$, where $\rho > 0$ is defined in \eqref{E:rho.inf.defn}. We show that $\gamma \notin \hat{B}(Y')$. By \eqref{E:L1.Y'}, \eqref{E:exprssn.for.f'.gamma}, and \eqref{E:rho.gamma.quo.defn},
	\begin{align*}
		L^{1}(\gamma, Y') - L^{1}(\beta, Y') &= 
		      -\sum_{i=1}^{n} \epsilon_{i} \,
		            \bigl( (\gamma_{0} - \beta_{0}) + x_{i}(\gamma_{1} - \beta_{1}) \bigr) 
		               + \delta (a - u \gamma_{1}) \\
		     &= f_{\gamma}'(0) - \delta (u \gamma_{1} -a) \\
		     &= (\rho_{\gamma} - \delta)(u \gamma_{1} -a) \\
		     &\geq (\rho - \delta)(u \gamma_{1} -a) > 0,
	\end{align*}
since $u \gamma_{1} > a$ and $\delta \in (0, \rho)$. 

Now, $\beta \in \hat{B}(Y')$, yet as we just have shown for $\delta$ sufficiently small, 
$\gamma \neq \beta$ implies that $\gamma \notin \hat{B}(Y')$. Therefore, the only option is to conclude that \eqref{E:B.Y'.=beta} holds. In particular $Y' \in \Y_{LAD}'$. The proposition is proved. 
\end{proof} 
  \begin{corly}   \label{C:nonunique.mins.are.sings}
Let $Y \in \Y$ and suppose $Y$ is not collinear. Then $Y$ is a singularity of LAD 
(w.r.t.\ $\Y_{LAD}'$) if and only if $\hat{B}(Y)$ contains more than one point 
(so $Y \notin \Y_{LAD}'$).
  \end{corly}
\begin{proof} Let $Y \in \Y$ be noncollinear. By \eqref{E:unique.LAD.soln.means.not.sing}, if $Y$ has just one LAD solution then it is not a singularity.

Suppose $\hat{B}(Y)$ is not a singleton but there exists $\bar{\beta}_{1}^{1 \times k}$ s.t.\ $\beta= (\beta_{0}, \beta^{T}_{1})^{T} \in \hat{B}(Y)$ implies $\beta_{1} = \bar{\beta}_{1}$. Then we get a contradiction to lemma \ref{L:basic.LAD.soln.facts}(b,i). Hence, there exist $\beta^{i} = (\beta^{i}_{0}, \beta^{iT}_{1})^{T} \in \hat{B}(Y)$ ($i = 1,2$) s.t.\ $\beta^{1}_{1} \neq \beta^{2}_{1}$. Let $U$ be as in lemma \ref{L:basic.LAD.soln.facts}(b). Then $U$ contains a basis for $\RR^{k}$. It follows that there exists $u \in U$ s.t.\ $u \beta^{1}_{1} \neq u \beta^{2}_{1}$. WLOG $u \beta^{1}_{1} < u \beta^{2}_{1}$. By lemma \ref{L:basic.LAD.soln.facts}b(ii), $-u \in U$. Therefore, if $\gamma = \beta = (\beta_{0}, \beta_{1}^{T})^{T}$ maximizes $u \gamma_{1}$ among all $\gamma = (\gamma_{0}, \gamma_{1}^{T})^{T} \in \hat{B}(Y)$ and $\gamma = \beta' = (\beta'_{0}, \beta_{1}^{'T})^{T}$ maximizes $-u \gamma_{1}$ among all $\gamma = (\gamma_{0}, \gamma_{1}^{T})^{T} \in \hat{B}(Y)$, then 
	\[
		u \beta'_{1} \leq u \beta^{1}_{1} < u \beta^{2}_{1} \leq u \beta_{1}.
	\]
Thus, $\beta'_{1} \neq \beta_{1}$. But from \eqref{E:B.Y'.=beta} in the proof of the proposition, we know that $Y$ can approximated arbitrarily well by data sets in $\Y_{LAD}'$ whose estimated coefficients are either $\beta$ or $\beta'$. Since $\beta'_{1} \neq \beta_{1}$, the LAD planes corresponding to these elements of $\Y_{LAD}'$ are different. Thus, as $Y'' \to Y$ through $\Y_{LAD}'$, the LAD plane $\Phi_{LAD}(Y'')$ does not converge in $G(k, \nvar)$. I.e., $Y$ is a singularity of LAD (w.r.t.\ $\Y_{LAD}'$).
\end{proof}

  \begin{lemma}  \label{L:diff.rank.condn.implications}
Suppose $Y = (X^{n \times k}, y^{n \times 1})$ is collinear but satisfies condition \eqref{E:diff.rank.condition}. Then:
	\begin{enumerate}
		\item $y \neq 0^{n \times 1}$ and the rows of $Y$ lie exactly 
		   on a unique $k$-plane. 
		          I.e., $Y \in \Pf^{k}$.
		             \label{I:Y.on.k-plane}
		\item If the rows of $X$ are mean-centered, 
		i.e., $1_{n} X = 0^{1 \times k}$, 
		      then the $k$-plane mentioned in statement \ref{I:Y.on.k-plane} passes 
		      through the origin and the point $(0^{1 \times k}, 1)$, the rank 
		        of $X$ is $k-1$, and the rank of $Y$ is $k$. 
		        \label{I:X.mean.centered.rank}
		\item $Y$ is \emph{not} a singularity of LAD. 
		\label{I:diff.rank.condn.implies.not.LAD.sing}
	\end{enumerate}
  \end{lemma}
  \begin{proof} Let $(x_{i}^{1 \times k}, y_{i})$ be the $i^{th}$ row of $Y$ ($i \in \NN_{n}$), for $i = 2, \ldots, n$, let $z_{i} := x_{i} - x_{1} \in \RR^{k}$, $v_{i} := y_{i} - y_{1} \in \RR$. Since $Y$ satisfies condition \eqref{E:diff.rank.condition}, if $2 \leq i_{1} < \ldots < i_{\nvar} \leq n$, 
then $(z_{i_{2}}, v_{i_{2}}), \ldots, (z_{i_{\nvar}}, v_{i_{\nvar}})$ are linearly independent. It follows that the $(n-1) \times \nvar$ matrix, $A$, whose $i^{th}$ row is $(z_{i+1}, v_{i+1})$ has rank at least $k$. 

On the other hand, since $Y$ is collinear, by \eqref{D:collinearity} the $(n-1) \times k$ matrix $Z$ whose $i^{th}$ row is $z_{i+1}$ has rank no greater than $k-1$. 
Thus, there exist $\alpha_{i_{2}}, \ldots, \alpha_{i_{\nvar}} \in \RR$, not all 0, s.t. 
$\sum_{j=2}^{\nvar} \alpha_{i_{j}} \, z_{i_{j}} = 0$. 
But $(z_{i_{2}}, v_{i_{2}}), \ldots, (z_{i_{\nvar}}, v_{i_{\nvar}})$ are linearly independent. Therefore,
	\begin{equation}    \label{E:sum.alpha.ij.zij.vij}
		\sum_{j=2}^{\nvar} \alpha_{i_{j}} \, (z_{i_{j}}, v_{i_{j}}) = (0, v) \in \RR^{\nvar}, 
		           \text{ where } v \neq 0.
	\end{equation}

WLOG in \eqref{E:sum.alpha.ij.zij.vij} we may assume $\alpha_{i_{q}} \neq 0$. Suppose the rank of $Z$ is less than $k-1$. Then $z_{i_{\nvar}}$ is a linear combination of $z_{i_{2}}, \ldots, z_{i_{\nvar-1}}$. Therefore, there exists $\gamma_{i_{j}} \in \RR$ ($j = 2, \ldots, \nvar-1$) s.t.\ 
$\sum_{j=2}^{\nvar-1} \gamma_{i_{j}} \, z_{i_{j}} = 0$. Hence,
	\begin{equation*}    
		\sum_{j=2}^{\nvar-1} \gamma_{i_{j}} \, (z_{i_{j}}, v_{i_{j}}) = (0, w) \in \RR^{\nvar}, 
		           \text{ where } w \neq 0.
	\end{equation*}
Let  $\beta_{i_{j}} := (v/w) \gamma_{i_{j}}$ ($j = 2, \ldots, \nvar-1$) and let 
$\beta_{i_{\nvar}} := 0$. Then 
	\begin{equation*}
		\sum_{j=2}^{\nvar} \beta_{i_{j}} \, (z_{i_{j}}, v_{i_{j}}) = (0, v) \in \RR^{\nvar} .
	\end{equation*}
where this $v$ is the same $v$ as in \eqref{E:sum.alpha.ij.zij.vij}. 
Since $\alpha_{i_{\nvar}} - \beta_{i_{\nvar}} = \alpha_{i_{\nvar}} \neq 0$ we certainly have that not all
$\alpha_{i_{2}} - \beta_{i_{2}}, \ldots, \alpha_{i_{\nvar}} -\beta_{i_{\nvar}}$ are 0. However,  
	\begin{equation*}
		\sum_{j=2}^{\nvar} (\alpha_{i_{j}} - \beta_{i_{j}}) \, (z_{i_{j}}, v_{i_{j}}) 
		  = (0, 0) \in \RR^{\nvar}.
	\end{equation*}
This contradicts the fact that $\bigl( (z_{i_{2}}, v_{i_{2}}), \ldots, (z_{i_{\nvar}}, v_{i_{\nvar}}) \bigr)$ are linearly independent. We conclude that $Z$ has rank $k-1$. 
Let $\zeta \subset \RR^{k}$ be the the row space, $\rho(Z)$, of $Z$, i.e., the span 
of $z_{2}, \ldots, z_{n}$, so $\zeta$ has dimension $k-1$.

Recall that $A^{(n-1) \times \nvar}$ is the matrix whose $i^{th}$ row is $(z_{i+1}, v_{i+1})$. It follows that the row space of $A$ is the $k$-plane $\xi := \zeta \times \RR$. (See \eqref{E:sum.alpha.ij.zij.vij}.) In particular, $(0,1) \in \xi$ and $y \neq 0^{n \times 1}$. \emph{A fortiori}, by letting $\alpha_{j} = 0$ if $j \notin \{ i_{2}, \ldots, i_{\nvar} \}$, we may assume that there are $\alpha_{2}, \ldots, \alpha_{n}  \in \RR$ s.t.\ 
	\begin{equation}  \label{E:sum.zi.vi.=.0.v}
		\sum_{j=2}^{n} \alpha_{j} \, (z_{j}, v_{j}) = (0, v) \in \RR^{\nvar}, 
		     \text{ where } v \neq 0 \text{ is arbitrary.}
	\end{equation}

Moreover, by \eqref{E:diff.rank.condition}, for \emph{any} 
$1 \leq i_{1} < \ldots < i_{\nvar} \leq n$, we have that 
$(z_{i_{2}}, v_{i_{2}}), \ldots, (z_{i_{\nvar}}, v_{i_{\nvar}})$ is a basis of $\xi$. But $\nvar-1 = k$. Hence, $\xi$ is the only $k$-plane containing $(z_{i_{2}}, v_{i_{2}}), \ldots, (z_{i_{\nvar}}, v_{i_{\nvar}})$. Thus, the rows of $Y$ lie exactly on the $k$-plane $(x_{1}, y_{1}) + \xi$ and that is the only plane with that property. This proves statement \ref{I:Y.on.k-plane} in the lemma.

No longer require that $X$ be mean-centered. Now suppose $X$ is mean-centered. We show that $(x_{1}, y_{1})  \in \xi$ so the rows of $Y$ lie exactly on $\xi$. Since $X$ is mean-centered,
	\begin{equation*}  \label{E:x1.in.terms.zj}
		x_{1} = - n^{-1} \sum_{j=2}^{n} z_{j}.
	\end{equation*}
By \eqref{E:sum.zi.vi.=.0.v}, there exist $\alpha_{2}, \ldots, \alpha_{\nvar} \in \RR$ s.t.\   
	\begin{equation*}
		\sum_{j=2}^{n} \alpha_{j} \, (z_{j}, v_{j}) 
			= \left( 0, \, y_{1} + n^{-1} \sum_{j=2}^{n} v_{j} \right).
	\end{equation*}
Therefore,  
	\begin{equation}  \label{E:x1.y1.in.trms.others}
	   (x_{1}, y_{1}) = \sum_{j=2}^{n} (\alpha_{j} - n^{-1}) \, (z_{j}, v_{j}) \in \xi.
	\end{equation}
In particular, $(x_{1}, y_{1}) + \xi = \xi$. 

\eqref{E:x1.y1.in.trms.others} implies that $x_{1} \in \zeta$. But $x_{j} = z_{j} + x_{1}$ 
so $\zeta$ is the row space of $X$. It follows that $\text{rank} \, X = k-1$. Moreover, \eqref{E:x1.y1.in.trms.others} implies 
that $(x_{1}, y_{1}) - (x_{1}, y_{1}) = 0 \in (x_{1}, y_{1}) + \xi$.

Consider the matrix $Y_{1}$, whose first row is $(x_{1}, y_{1})$ and whose remaining rows 
are $(x_{2}, y_{2}) - (x_{1}, y_{1}), \ldots, (x_{n}, y_{n}) - (x_{1}, y_{1})$. $Y_{1}$ has the same rank as $Y$. But, by \eqref{E:diff.rank.condition}, $rank \, Y_{1} \geq k$. 
Therefore, $rank \, Y \geq k$. But since $rank \, X = k-1$, we must have $rank \, Y = k$. (This elaborates upon statement \ref{I:Y.on.k-plane} in the lemma.)
This proves statement \ref{I:X.mean.centered.rank} in the lemma. 

Suppose $\{ Y_{m} \}  \subset \Y_{LAD}'$ converges to $Y$. (Since $\Y_{LAD}'$ is dense 
in $\Y$ -- proposition \ref{P:Y'.LAD.is.dense} -- , such a sequence exists.) 
Given $m = 1, 2, \ldots$, let $\xi_{m} \in G(k,\nvar)$ be the $k$-dimensional subspace 
of $\RR^{\nvar}$ parallel to the (unique) LAD plane of $Y_{m}$. 

We prove that $\xi_{m} \to \xi$ as $m \to \infty$. Let the rows of $Y_{m}$ be $(x_{m1}, y_{m1}), \ldots, (x_{mn}, y_{mn})$. Since $Y_{m} \in \Y_{LAD}'$, by \eqref{E:Y.LAD'.defn}, we have that $Y_{m}$ is \emph{not} collinear and there is only one LAD plane for $Y_{m}$, so, by lemma \ref{L:basic.LAD.soln.facts}(b,iii), taking a subsequence if necessary, we may assume there exist \emph{fixed} indices $1 \leq i_{1} < \ldots < i_{\nvar} \leq n$ s.t.\ for every $m$, we have 
$z_{mi_{j}} := x_{mi_{j}} - x_{mi_{1}}$ ($j = 2, \ldots , \nvar$) is a basis for $\RR^{k}$. Moreover, if $v_{mi_{j}} := y_{mi_{j}} - y_{mi_{1}}$ ($j = 2, \ldots , \nvar$) then $\xi_{m}$ passes through $(z_{mi_{j}}, v_{mi_{j}})$ ($j = 2, \ldots , \nvar$). But $z_{mi_{j}} \to z_{i_{j}}$ and $v_{mi_{j}} \to v_{i_{j}}$ and $(z_{i_{2}}, v_{i_{2}}), \ldots, (z_{i_{\nvar}}, v_{i_{\nvar}})$ is a basis of $\xi$, by \eqref{E:diff.rank.condition}. Hence, by \eqref{E:convergence.in.Grassmann}, $\xi_{m} \to \xi$. Thus, $Y$ is not a singularity. Hence, statement \ref{I:diff.rank.condn.implies.not.LAD.sing} of the lemma is proved.
  \end{proof}
  
	\begin{prop} \label{P:few.collin.LAD.sings}
Every singularity of LAD w.r.t.\ $\Y_{LAD}'$ in $\Pf^{k}$ is collinear. 
The dimension of the space of all collinear singularities of LAD is $\leq (n+1)k -1$. Hence, by lemma \ref{L:dim.set.of.collin.data.sets}, the codimension of the set of collinear singularities of LAD in the space of all collinear data sets is at least 1. 
	\end{prop}

In the proof of the proposition we make use of the following two lemmas. Recall the definition of angle, \eqref{E:angle.between.vectors}. 

  \begin{lemma}  \label{L:u.v.x.ineq}
Let $u,v,x \in \RR^{k}$ and suppose $|u| = |v| = 1$. Then
	\begin{equation} \label{E:u.v.x.ineq}
		\bigl| (u \cdot x) u - (v \cdot x) v \bigr| \leq |x| \angle (u, v),
	\end{equation}
where $\angle (u, v)$ is the angle between $u$ and $v$. Moreover, \emph{the LHS of \eqref{E:u.v.x.ineq} only depends on $x$ through the length of its projection onto the plane spanned by $u$ and} $v$.
  \end{lemma}
Recall that by \eqref{E:angle.between.vectors}, $\angle (u, v) \leq \pi$. The last sentence is easy to confirm when $u \cdot v = 0$. 
  \begin{proof}
WLOG assume $x$ lies in the plane spanned by $u$ and $v$ and $|x| = 1$. Choose orthogonal coordinates for that span so that $u = (1,0)$, $v = (\cos \theta, \sin \theta)$, and $x = (\cos \alpha, \sin \alpha)$, where $-\pi \leq \alpha, \theta \leq \pi$. 
Then $\angle (u, v) = |\theta|$. The square of the LHS of \eqref{E:u.v.x.ineq} is then
\begin{align*}
  \cos^{2} \alpha - &2 \cos \alpha  
        (\cos \theta \cos \alpha + \sin \theta \sin \alpha) \cos \theta \\
         & \qquad \qquad + (\cos \theta \cos \alpha + \sin \theta \sin \alpha)^{2} \\
         &= \cos^{2} \alpha - 2 \cos \alpha \cos (\theta - \alpha) \cos \theta 
           + \cos^{2} (\theta - \alpha) \\
        &= \cos^{2} \alpha - \cos (\theta - \alpha) \bigl( 2 \cos \alpha \cos \theta
           - \cos (\theta - \alpha) \bigr) \\
        &= \cos^{2} \alpha - \cos (\theta - \alpha) 
          \bigl( 2 \cos \alpha \cos \theta - \cos \alpha \cos \theta - \sin \alpha \sin \theta  \bigr) \\
        &= \cos^{2} \alpha - \cos (\theta - \alpha) 
          \bigl( \cos \alpha \cos \theta  - \sin \alpha \sin \theta  \bigr) \\
        &= \cos^{2} \alpha - \cos (\theta - \alpha) 
          \cos (\theta + \alpha) \\
        &= \cos^{2} \alpha - \bigl[ \cos (\theta - \alpha) \cos (\theta + \alpha) 
          + \sin (\theta - \alpha) \sin (\theta + \alpha)  \bigr] \\
        & \qquad \qquad + \sin (\theta - \alpha) \sin (\theta + \alpha) \\
        &= \cos^{2} \alpha - \cos 2\alpha + \sin (\theta - \alpha) \sin (\theta + \alpha) \\
        &= \cos^{2} \alpha - (\cos^{2} \alpha - \sin^{2} \alpha)
           + \sin (\theta - \alpha) \sin (\theta + \alpha) \\
        &= \sin^{2} \alpha + \sin (\theta - \alpha) \sin (\theta + \alpha) \\
        &= \sin^{2} \alpha + (\sin \theta \cos \alpha - \cos \theta \sin \alpha)
          (\sin \theta \cos \alpha + \cos \theta \sin \alpha) \\
        &= \sin^{2} \alpha + \sin^{2} \theta \cos^{2} \alpha 
          - \cos^{2} \theta \sin^{2} \alpha  \\
        &= \sin^{2} \alpha + \sin^{2} \theta \cos^{2} \alpha 
          - (1 - \sin^{2} \theta) \sin^{2} \alpha  \\
        &= \sin^{2} \alpha + \sin^{2} \theta \cos^{2} \alpha 
          - \sin^{2} \alpha + \sin^{2} \alpha \sin^{2} \theta)  \\
        &= (\sin^{2} \theta) (\cos^{2} \alpha + \sin^{2} \alpha)  \\
        &= \sin^{2} \theta . 
\end{align*}
But $|\sin \theta| \leq |\theta| = \angle (u, v)|$.
  \end{proof}

   \begin{lemma}  \label{L:dim.of.low.rank.mats}
Let $\mu, \nu = 1, 2, 3, \ldots$ and let $\mcl{M}$ be the set of $\mu \times \nu$ real matrices. Let $r < \min \{ \mu, \nu \}$ be a non-negative integer 
and let $\mcl{M}_{r} = \{ M \in \mcl{M} : \text{rank } M = r \}$. Let 
$\mcl{M}_{0:r} = \{ M \in \mcl{M} : \text{rank } M \leq r \} = \bigcup_{s=0}^{r} \mcl{M}_{s}$. 
Then $\mcl{M}_{r}$ and  $\mcl{M}_{0:r}$ are closed subsets of $\mcl{M}$ and 
    \begin{equation}  \label{E:r.mu+r.nu-r.sqrd}
      \dim \mcl{M}_{0:r} = \dim \mcl{M}_{r} = r \mu + r \nu - r^{2} < \mu \nu .
    \end{equation} 
In particular, 
$\mcl{M}_{0:r}$ has empty interior relative to $\mcl{M}$.
   \end{lemma}
Notice that the quantity $r \mu + r \nu - r^{2}$ is an increasing function 
of $r < \min \{ \mu, \nu \}$.
  \begin{proof} (See lemma \ref{L:rank.k.mats.form.manif}.) By lemma \ref{L:rank.lwr.semicont}, $\mcl{M}_{r}$ is closed. 

A matrix of rank $r = 0$, is just the zero matrix. I.e., $\mcl{M}_{0} = \{ 0^{\mu \times \nu} \}$. Thus, $\dim \mcl{M}_{0} = 0$ and \eqref{E:r.mu+r.nu-r.sqrd} is satisfied in that case.

Let $ r \in \bigl[1, \min \{ \mu, \nu \} \bigr)$, let $s = 1, \ldots, r$, 
and let $0 < i_{1} < \cdots < i_{s} \leq \nu$ be integers. 
Let $0 < j_{1} < \cdots < j_{\nu-s} \leq \nu$ be the remaining integers between 1 and $\nu$ inclusive. If $A$ and $B$ are $\mu \times s$ and $s \times (\nu-s)$ matrices, resp., 
let $f(A,B) := f_{i_{1}, \ldots, i_{s}}(A,B)$, be the $\mu \times \nu$ matrix whose $(i_{m})^{th}$ column is the $m^{th}$ column of $A$ ($m = 1, \ldots, s)$ and whose $(j_{m})^{th}$ column is 
the $m^{th}$ column of $AB$ ($m = 1, \ldots, \nu - s$). 

Assume $A$ has full rank $s$. For simplicity, for the moment assume $i_{m} = m$ 
($m=1, \ldots, s$). Then $f_{1, \ldots, s}(A,B) = A (I_{s} , B)$, where $I_{s}$ is the $s \times s$ identity matrix. It follows that the rank of $f(A,B)$ is $s$. Conversely, it is clear that any matrix 
in $\mcl{M}_{s}$ equals $f_{i_{1}, \ldots, i_{s}}(A,B)$ for some $(A,B)$ and some choice 
of $i_{1} < \cdots < i_{s}$. By \eqref{E:dim.Kk=k.nvar}, the dimension of the space of such $A$'s is $\mu s$. The dimension of the space of such $B$'s is obviously $s \times (\nu-s)$. 

Note that $f_{i_{1}, \ldots, i_{s}}$ is invertible: Knowing $i_{1}, \ldots, i_{s}$, we can read off $A$ and $W := AB$ from $f(A,B)$. Since $A$ has full rank, we have $B = (A^{T} A)^{-1} A^{T}  W.$

By example \ref{Ex:ratnl.fns.loc.Lip}, $f(A,B)$ is Lipschitz in $(A, B)$ and $f^{-1}$ is locally Lipschitz. (Relevant to the $f^{-1}$ case is Lang \cite[Proposition 8, p.\ 334]{sL65.Algebra}.)
The domain of $f$ is the set, $\mcl{AB}_{s}$, of all pairs $(A,B)$ just specified. Thus, by \eqref{L:loc.Lip.image.of.null.set.is.null}, $\dim f(\mcl{AB}_{s}) = \dim \mcl{AB}_{s}.$
By lemma \ref{L:dim.of.product} or Boothby \cite[Theorem 1.7, p.\ 57]{wmB75}, the Hausdorff dimension of $\mcl{AB}_{s}$ is $s \mu + s \nu - s^{2}$. 
Hence, $\dim f(\mcl{AB}_{s}) = s \mu + s \nu - s^{2}$.

As $s = 1, \ldots, r$ and $i_{1}, \ldots, i_{s}$ vary, the images of the corresponding $f$'s, together with $\mcl{M}_{0}$, cover $\mcl{M}_{r}$. Therefore, by \eqref{E:dim.of.whole.=.max.dim.of.parts},
$\dim \mcl{M}_{0:r} = \max_{s = 0, \ldots, r} (s \mu + s \nu - s^{2}) = r \mu + r \nu - r^{2}$.
  \end{proof}

\begin{proof}[Proof of proposition \ref{P:few.collin.LAD.sings}] 
Let $Y = (X^{n \times k}, y^{n \times 1}) \in \Pf^{k}$ be a singularity of LAD. Then, by \eqref{E:not.coll.in.Pk.not.LAD.sing}, $Y$ must be collinear. This is just the first sentence in the proposition. 

By lemma \ref{L:diff.rank.condn.implications}\eqref{I:diff.rank.condn.implies.not.LAD.sing}, for such $Y$ condition \eqref{E:diff.rank.condition} must fail. We now prove that the Hausdorff dimension of the collection of all collinear data sets for which condition \eqref{E:diff.rank.condition} fails is no greater than $(n+1)k-1$. The second sentence of the proposition will then follow. Since, by lemma \ref{L:dim.set.of.collin.data.sets}, the dimension of the set of collinear data sets is $(n+1)k$, the last sentence of the proposition will hence follow. 

Let $(x_{i}, y_{i})$ be the $i^{th}$ row of $Y$ ($i \in \NN_{n}$). Since condition \eqref{E:diff.rank.condition} fails, there exist  $1 \leq i_{1} < \ldots < i_{\nvar} \leq n$  s.t.\  
	\begin{equation}  \label{E:(x-x1.y-y1).lin.dep}
		(x_{i_{2}} - x_{i_{1}}, y_{i_{2}} - y_{i_{1}}), 
		  (x_{i_{3}} - x_{i_{1}}, y_{i_{3}} - y_{i_{1}}),\ldots, 
		    (x_{i_{\nvar}} - x_{i_{1}},  y_{i_{\nvar}} - y_{i_{1}}) 
		\text{ are linearly dependent.}	
	\end{equation}
Let $J = \{ i_{1}, \ldots, i_{\nvar} \}$ and let 
    \begin{equation*}
      \mcl{Q}_{J}  \text{ denote the set of all collinear data sets in } \Pf^{k} 
        \text{ satisfying \eqref{E:(x-x1.y-y1).lin.dep}}. 
    \end{equation*}
We show that $\mcl{Q}_{J}$ has Hausdorff dimension $\leq (n+1)k-1$.

Let $\omega \in G(k-1,k)$. Let $e_{1}, \ldots, e_{k}$ be an orthonormal basis 
of $\RR^{k}$. Let $u \in \RR^{k}$ be a unit vector perpendicular to $\omega$. 
For some $i$ we have $u \cdot e_{i} \neq 0$. (``$\cdot$'' indicates the usual inner product in $\RR^{k}$.)  I.e., $e_{i} \notin \omega$. Let $i \in \NN_{k}$ and
    \begin{equation*}
      W_{i} := \bigl\{ \omega \in G(k-1,k) : e_{i} \notin \omega \bigr\} .
    \end{equation*}

\emph{Claim:} 
    \begin{equation}  \label{E:Wi.is.open}
      W_{i} \text{ is open in } G(k-1,k).
    \end{equation} 
If $\omega \in G(k-1,k)$, 
let $\Pi_{\omega} := \Pi(\omega)^{k \times k}$ be the matrix of orthogonal projection 
onto $\omega$. By lemma \ref{L:proj.mat.is.imbedding.of.Grass} (with $(k-1,k)$ in place 
of $(k, \nvar)$) $\Pi$ is smooth. Hence, $g : \omega \mapsto | \Pi_{\omega} e_{i} |$ is a continuous map of $G(k-1,k)$ into $[0, \infty)$. But $W_{i} = g^{-1}[0, 1)$. This proves the claim \eqref{E:Wi.is.open}.

Let $\{ \ell_{1}, \ldots, \ell_{k-1} \} = \{ 1, \ldots, k \} \setminus \{ i \}$. 
For $\omega \in G(k-1,k)$, let $Z(\omega) := Z_{i}(\omega)^{(k-1) \times k}$ be the matrix whose $j^{th}$ row is $\Pi_{\omega} e_{\ell_{j}}$ ($j = 1, \ldots, k-1$). 
Thus, $\rho \bigl[ Z_{i}(\omega)) \bigr] \subset \omega$. 

Let $\omega \in W_{i}$. So $e_{i} \notin \omega$. \emph{Claim:}
    \begin{equation}  \label{E:rank.Z(omega)=k-1}
      rank \, Z(\omega ) = k-1, 
    \end{equation} 
For suppose $rank \, Z(\omega ) < k-1$. Then there are numbers $a_{1}, \ldots, a_{k-1}$, not all 0, s.t.\  
	\begin{equation}  \label{E:pi.omega.sum.=.0}
		\Pi_{\omega} \left( \sum_{j=1}^{k-1} a_{j} e_{\ell_{j}} \right) 
		   = \sum_{j=1}^{k-1} a_{j} \Pi_{\omega} e_{\ell_{j}} = 0.
	  \end{equation}
Let $x := \sum_{j=1}^{k-1} a_{j} e_{\ell_{j}}$. Then $x \neq 0$, since $e_{\ell_{j}}$ ($j = 1, \ldots, k-1$) are linearly dependent. By \eqref{E:pi.omega.sum.=.0}, we have $x  \perp \omega$. Thus, the $(k-1)$ dimensional subspace, $x^{\perp}$, of $\RR^{k}$ that is orthogonal to $x$ contains $\omega$. But $\dim \omega = k-1$. I.e., $x^{\perp} = \omega$. But $e_{1}, \ldots, e_{k}$ are  orthonormal and by definition of $\{ \ell_{1}, \ldots, \ell_{k-1} \}$ we have $x \perp e_{i}$. I.e, $e_{i} \in \omega$, contradicting the definition of $W_{i}$. The  claim \eqref{E:rank.Z(omega)=k-1} follows.

We have already observed that 
$\rho \bigl[ Z_{i}(\omega)) \bigr] \subset \omega$. Since $\dim \omega = k-1$, we have
    \begin{equation} \label{E:row.space.of.Z=omega}
       \rho \bigl[ Z_{i}(\omega)) \bigr] = \omega .
    \end{equation}

Let $\xi, \zeta \in G(k-1,k)$ and let $u, v \in \RR^{k}$ be unit vectors perpendicular 
to $\xi, \zeta$, resp., s.t.\ $u \cdot v \geq 0$. Define the distance between $\xi$ 
and $\zeta$ to be the angle, $\angle(u,v)$, between $u$ and $v$. (See remark \ref{R:metrics.on.P(S).and.G(k,k+1)} and \eqref{E:metric.on.G(k,k+1)}.) 
\emph{Claim:} $Z(\omega)$ is Lipschitz in $\omega \in W_{i}$ (w.r.t.\ $\angle$ and the Frobenius norm, \eqref{E:matrix.norm}; see lemma \ref{L:proj.mat.is.imbedding.of.Grass} and appendix \ref{Chptr:Lip.Haus.meas.dim}; remember, $\omega \in W_{i}$ implies 
$e_{i} \notin \omega$). Let $\omega, \zeta \in W$ 
and let $u, v \in \RR^{k}$ be unit vectors perpendicular to $\omega, \zeta$, resp., 
s.t.\ $u \cdot v \geq 0$. Then $\Pi_{\omega} x = x - (u \cdot x)u$, $x \in \RR^{k}$. Similarly for $\Pi_{\zeta} x$. $\Pi_{\omega} e_{\ell_{j}}$ and $\Pi_{\zeta} e_{\ell_{j}}$ are the $j^{th}$ rows of $Z(\omega)$ and $Z(\zeta)$, resp. Thus, by lemma \ref{L:u.v.x.ineq}, 
	\[
	    \| \Pi_{\omega} e_{\ell_{j}} -\Pi_{\zeta} \| \leq 
		  \sum_{j=1}^{k-1} \bigl| \Pi_{\omega} e_{\ell_{j}} 
		    -\ Pi_{\zeta} e_{\ell_{j}} \bigr| \leq (k-1) \angle (u, v). 
		  \quad j = 1, \ldots, k-1.
	\]
But $\angle (u,v)$, the angle between $u$ and $v$, is the distance between $\omega$ 
and $\zeta$ and This proves the claim that $Z(\omega)$ is Lipschitz in 
$\omega \in W_{i}$. 

Now, $W_{i}$ is an open subset of $G(k-1,k)$ and so, 
by Milnor and Stasheff \cite[Lemma 5.1, p.\ 57]{jwMjdS74} (their notation differs slightly from ours) and \eqref{E:Haus.dim.s-manif.=.s},
    \begin{equation}  \label{E:dim.W.=.k-1}
        \dim W_{i} = k-1 .
    \end{equation}
    
Let $m \in \{ k, k+1, \ldots \}$. 
If $C$ is a $m \times k$ matrix of rank $< k,$ then its rows lie on some, perhaps not unique, $(k-1)$-dimensional subspace $\omega$ of $\RR^{k}$. 
Suppose $\omega \in W_{i}$.
Then, by \eqref{E:row.space.of.Z=omega}, the rows of $Z_{i}(\omega)$ form a basis 
of $\omega$. Thus, 
    \begin{multline}   \label{E:C=AZ.A.unique}
      \text{We can write } C = A Z_{i}(\omega), \text{ for some } i = 1, \ldots, k; \, 
        \omega \in G(k-1,k); \\
          \text{ and a unique } m \times (k-1) \text{ matrix } A.
    \end{multline}

Returning to the matrix $Y$, WLOG suppose \eqref{E:(x-x1.y-y1).lin.dep} holds 
with $i_{j} = j$ ($j =  1, \ldots, \nvar$) so $J = \NN_{\nvar} := \{1, \ldots, \nvar \}$. 
Let $z_{i} = x_{i} - x_{1} \in \RR^{k}$ and $u_{i} = y_{i} - y_{1} \in \RR$ ($i=2, \ldots, q$; see \eqref{E:initially.in.lin.reg.nvar=k+1}). Since $Y$ is collinear by assumption, by definition \ref{D:collinearity}, we have that 
$z_{2}, \ldots, z_{\nvar}$ do not span $\RR^{k}$. There are two ways \eqref{E:(x-x1.y-y1).lin.dep} can be true. 

\emph{The first way \eqref{E:(x-x1.y-y1).lin.dep} can be true} is if 
	\begin{multline} \label{E:z1...znvar.lin.indep}
		z_{2}, \ldots, z_{\nvar} \text{ span a space of dimension } k-1.  \\
		         \text{ It follows there exists } w \in \RR^{k} \text{ s.t.\ } 
		                  u_{i} = w \cdot z_{i} \qquad (i = 2, \ldots, \nvar).
	\end{multline}
We provide a local parametrization of data sets for which that is true. (The parametrization spills over to include data sets for which \eqref{E:z1...znvar.lin.indep} is false.) $Y$ is collinear but also in $\Pf^{k}$. Therefore, by \eqref{E:when.is.Y.in.Pk}, we have that 
$rank \, (Y - 1^{n} w_{1}^{T} Y) = k$. 
Here, 
    \begin{equation*}
      w_{1}^{T} = (1, 0, \ldots, 0)^{1 \times n} .
    \end{equation*} 
Thus, all the rows 
of $1^{n} w_{1}^{T} Y$ equal $(x_{1}, y_{1})^{1 \times \nvar}$.
Write $Y_{1} := Y - 1^{n} w_{1}^{T} Y$. Then the first row of $Y_{1}$ is 0. The submatrix consisting of rows $2, \ldots, \nvar$ of $Y_{1}$ has the form $(M, u)$, where $M^{k \times k}$ has rows $z_{2}, \ldots, z_{\nvar}$ and $u^{k \times 1} = (u_{2}, \ldots, u_{\nvar})$. By \eqref{E:z1...znvar.lin.indep}, $M$ has rank $k-1$. Therefore, by \eqref{E:C=AZ.A.unique}, there exists $i = 1, \ldots, k$; $\omega \in G(k-1,k)$; and unique $A^{k \times (k-1)}$ s.t.\ $M = A Z_{i}(\omega)$. 

Let $N^{(n-\nvar) \times k}$ be the submatrix of $Y_{1}$ consisting of the first $k$ columns of the last $n-\nvar$ rows. Then
    \begin{equation*}
      X - 1^{n} x_{1} =
      \begin{pmatrix}
           0^{1 \times k} \\
           A Z_{i}(\omega) \\
           N
      \end{pmatrix} .
    \end{equation*}
Since $Y$ is collinear, by definition \ref{D:collinearity}, $X - 1^{n} x_{1}$ has rank $< k$. By \eqref{E:z1...znvar.lin.indep}, $A$ must be of full rank $k-1$. Hence, 
$rank \, (X - 1^{n} x_{1}) = k-1$. Thus, the row space $\rho(N)$ must lie in that 
of $Z_{i}(\omega)$. Hence, there exists $B^{(n-k-1) \times (k-1)}$ s.t.\ 
$N = B Z_{i}(\omega)$. (Recall that in section \ref{SS:LAD} $n-k-1 = n-\nvar$.) The last $n - \nvar$ entries in the last column of $Y_{1}$ constitute a column vector 
$b^{(n-k-1) \times 1}$. 

Since the rank of $Y_{1}$ is $k$ but $rank \, (X - 1^{n} x_{1}) = k-1$, the last column of $Y_{1}$, \emph{viz.}\ $y - 1^{n} y_{1}$, does not lie in the column space 
of $X - 1^{n} x_{1}$. But the first element of $y - 1^{n} y_{1}$ is 0 and the $k$ elements after that constitute a vector in the column space of $M = A Z_{i}(\omega)$. It follows that the remaining $n - \nvar$ entries in $y - 1^{n} y_{1}$, forming the vector $b$, cannot belong to the column space of $N = B Z_{i}(\omega)$.

Putting this all together, we get the desired parametrization. Let $v$ be a $1 \times \nvar$ vector (corresponding to $(x_{1}, y_{1})^{1 \times \nvar}$); $A$, a $k \times (k-1)$ matrix \emph{of full rank;} $a$, a $(k-1) \times 1$ vector; $b$, a $(n-k-1) \times 1$ vector; $B$, an $(n-k-1) \times (k-1)$ matrix; $i = 1, \ldots, k$; and $\omega \in W_{i}$. A vector in the column space of $A Z_{i}(\omega)$ can be written $A Z_{i}(\omega) c^{T}$, where $c^{1 \times k} \in \rho(Z_{i}(\omega)$. Such a $c$ can be written 
$c = (a^{(k-1) \times 1})^{T} Z_{i}(\omega)$. Define an $n \times \nvar$ matrix by,
	\begin{equation*}
		f_{i}(x_{1}, A, a, b, B, \omega) =
			\begin{pmatrix}
			            0^{1 \times \nvar} \\
				\begin{matrix}
					AZ_{i}(\omega) & AZ_{i}(\omega) Z_{i}(\omega)^{T} a  \\
					BZ_{i}(\omega) & b
				\end{matrix}
			\end{pmatrix}
		+
			\begin{pmatrix}
				 v \\
				 v \\
				 \vdots \\
				 v
			\end{pmatrix}^{n \times \nvar}
		.
	\end{equation*}
As we have seen, any collinear $Y \in \Pf^{k}$ for which \eqref{E:z1...znvar.lin.indep} holds is in the image of $f_{i}$ for some $i = 1, \ldots, k$. By \eqref{E:comp.of.Lips.is.Lip} and corollary \ref{C:cont.diff.=.loc.Lip}, $f_{i}$ is  locally Lipschitz (appendix \ref{Chptr:Lip.Haus.meas.dim}). 

By lemma \ref{L:rank.lwr.semicont}, $A$ varies over an open subset of a manifold. By \eqref{E:Wi.is.open}, so does $\omega$. The other arguments to $f_{i}$ clearly have the same property. It follows from lemma \ref{L:dim.of.product} or Boothby \cite[Theorem 1.7, p.\ 57]{wmB75}, that the Hausdorff dimension of the domain of $f_{i}$ is the sum of the dimensions of its factors. 

By \eqref{E:dim.Kk=k.nvar}, the space of $A^{k \times (k-1)}$ of full rank has dimension is $k(k-1)$, same as the dimension of the space of all $k \times (k-1)$ matrices. By \eqref{E:dim.W.=.k-1}, $\dim W_{i} = k-1$. Allow $b$ to be any $n-k-1$ vector, not just those not in the column space of $B Z_{i}(\omega)$. I.e., the dimension of the domain of  $f_{i}$ is,
	\begin{multline}   \label{E:fi.domain.dim}
		 \underbrace{(k+1)}_{v} + \underbrace{k(k-1)}_{A} + \underbrace{(k-1)}_{a} + 
		      \underbrace{(n-k-1)}_{b} + \underbrace{(n-k-1)(k-1)}_{B} 
		        + \underbrace{(k-1)}_{\omega}  \\
		     = nk+k-1. 
	\end{multline}
Therefore, by lemma \ref{L:loc.Lip.image.of.null.set.is.null}, the Hausdorff dimension of the image of $f_{i}$ -- \emph{and hence of the set of $Y$'s satisfying} \eqref{E:z1...znvar.lin.indep} -- is no larger than $nk+k-1$. By lemma \ref{L:dim.set.of.collin.data.sets}, this is 1 less than the dimension of the set of all collinear data sets. 

\emph{The second way \eqref{E:(x-x1.y-y1).lin.dep} can be true} is if $k > 1$ and 
	\begin{equation}   \label{E:zs.not.lin.indep}
		\text{the span of }z_{2}, \ldots, z_{\nvar} \text{ in } 
		  \RR^{k} \text{ has dimension } < k-1.
	\end{equation}
We provide a local parametrization of data sets for which that is true (and some for which it is not). Let $v$ be a $1 \times \nvar$ vector; $A$ a $k \times (k-1)$ matrix \emph{of rank $< k-1$}; $c$ a $k \times 1$ vector; $b$ a $(n-k-1) \times 1$ vector; $B$ a $(n-k-1) \times (k-1)$ matrix; $i = 1, \ldots, k$; and $\omega \in W_{i}$. 
Define
	\begin{equation*}
		g_{i}(v, A, c, b, B, \omega) =
			\begin{pmatrix}
			                         0^{1 \times \nvar} \\
				\begin{matrix}
					AZ_{i}(\omega) & c  \\
					BZ_{i}(\omega) & b
				\end{matrix}
			\end{pmatrix}
		+
			\begin{pmatrix}
				 v \\
				 v \\
				 \vdots \\
				 v
			\end{pmatrix}^{n \times \nvar} .
	\end{equation*}
It is easy to see that any collinear $Y \in \Pf^{k}$ for which \eqref{E:zs.not.lin.indep} holds is in the image of $g_{i}$ for some $i = 1, \ldots, k$.  

Then, by lemma \ref{L:dim.of.low.rank.mats} and \eqref{E:dim.W.=.k-1}, the Hausdorff dimension of the image of $g_{i}$ is no larger than 
	\begin{multline*}
		\underbrace{k+1}_{v} + \underbrace{(k-2)k+(k-2)(k-1)-(k-2)^{2}}_{A} 
		   + \underbrace{k}_{c} + \underbrace{(n-k-1)}_{b} \\
		     + \underbrace{(n-k-1)(k-1)}_{B}  + \underbrace{(k-1)}_{\omega} 
		          = nk+k-2.
	\end{multline*}
As before, by lemma \ref{L:loc.Lip.image.of.null.set.is.null}, the Hausdorff dimension of the image of $f_{i}$ is no larger than $nk+k-2$. By lemma \ref{L:dim.set.of.collin.data.sets}, this is 2 less than the dimension of the set of all collinear data sets.

But $\mcl{Q}_{J}$ is covered by a finite union of images of functions like $f_{i}$ 
or $g_{i}$. Allowing $J$ to vary, we get a finite cover of the set of all collinear data sets 
$Y \in \Y$ for which condition \eqref{E:diff.rank.condition} fails. Apply \eqref{E:dim.of.whole.=.max.dim.of.parts} to \eqref{E:fi.domain.dim} and the preceding.
  \end{proof}

  \begin{corly} \label{C:defective.collin.data.sets.dim}
The Hausdorff dimension of the collection of all collinear data sets for which condition \eqref{E:diff.rank.condition} fails is no greater than $(n+1)k-1$. Thus, the codimension, in the set of all collinear data sets, of the collection of all collinear data sets for which condition \eqref{E:diff.rank.condition} fails is at least 1.
  \end{corly}

\numberwithin{equation}{section}

\chapter{Neighborhood of $\Pf_{1}$ Fibered by Cones in Resistant Location Problem on the Circle}  \label{Chptr:rob.loc.circle.cones.appendix}

In this appendix we show how to fiber a neighborhood of $\Pf_{1}$ (see \eqref{E:Pk.defn.robust.loc.on.sphere}) in measuring location on a circle ($\nvar = 1$) by cones as in definition \ref{D:fibering.by.cones}. (This appendix isn't done but I like to think there's not much more to do on it.)

Recall \eqref{E:ma:nvar=1}.

\section{Geodesics in $\D := (S^{1})^{n}$} \label{S:Geods.in.D.2}
In this section we adopt the convention of using boldface to indicate vectors. Unboldened characters with subscripts will be the coordinates, i.e.\ components, of the vector. For example, if $\blds{\phi} \in \RR^{n}$ then 
$\blds{\phi} := (\phi_{1}, \ldots, \phi_{n})$ with $\phi_{1}, \ldots, \phi_{n} \in \RR$. (In truth, sometimes I fail to follow this convention.) 

Here we consider the problem of finding location on a circle. In this case $\D = \bigl( S^{1} \bigr)^{n} \subset \RR^{2n}$, where in accordance with \eqref{E:n>2,nvar>0}, $n > 2$. 

We orient the circle $S^{1}$ so that clockwise and counter clockwise has meaning for it. As usual counter clockwise is the positive direction. 

We adopt the convention of using boldface to indicate vectors. Unboldened characters with subscripts will be the coordinates, i.e.\ components, of the vector. For example, if $\blds{\phi} \in \RR^{n}$ then 
$\blds{\phi} := (\phi_{1}, \ldots, \phi_{n})$ with $\phi_{1}, \ldots, \phi_{n} \in \RR$. (In truth, sometimes I fail to follow this convention.)

Recall \eqref{E:directional.T.defn}: 
$\T := diag \, \D = \bigl\{ (y, \ldots, y) \in \D$, $y \in S^{1} \bigr\}$. 
$\D$ is covered by coordinate neighborhoods parametrized as follows. 
Recall \eqref{E:N.sub.n}. If $x \in \D$, then for some 
$\phi_{i} \in \RR, \, i \in \NN_{n}$ we have  
	\begin{equation} \label{E:wrap.defn.2}
		x = \wrap(\blds{\phi}) := ( \cos \phi_{1}, \sin \phi_{1}, 
		  \cos \phi_{2}, \sin \phi_{2}, 
			\ldots, \cos \phi_{n}, \sin \phi_{n} \bigr) \in (S^{1})^{n} 
	              \subset \RR^{2n} ,
	\end{equation}
where $\blds{\phi} := ( \phi_{1}, \ldots, \phi_{n} )$. Note that if 
$\blds{\phi}' \in \RR^{n}$ and $\wrap(\blds{\phi}') = \wrap(\blds{\phi})$ then there exists $\mbf{k} \in \ZZ^{n}$ 
s.t.\ $\blds{\phi}' = \blds{\phi} + 2 \mbf{k} \pi$. 

Since $\wrap$ is differentiable, by corollary \ref{C:cont.diff.=.loc.Lip},
it is locally Lipschitz. Since it is periodic, it is globally Lipschitz. 

Let 
    \begin{equation}  \label{E:complex.plane}
      \CC = \text{ the complex plane. }
    \end{equation}
In this paragraph $i$ will denote $sqrt{-1}$, the imaginary unit. The function 
$exp : \CC \to \CC$ given by $exp(x + iy) = e^{x} (\cos y + i \sin y)$ ($x,y \in \RR$) is not constant on any open subset of $\CC$. It's analytic. Therefore, by the Open Mapping Theorem (Rudin \cite[Theorem 10.32, p.\ 216]{wR66.realcmplx}), $exp$ is open. In particular, if $\phi \in \RR$ then the image of any (open) neighborhood of $i \phi$ in $\CC$ is open. In particular if $\delta > 0$ then 
$E := exp \bigl[ (-\delta, \delta) + i (-\delta, \delta) \bigr]$, where 
$(-\delta, \delta) + i (-\delta, \delta) = \bigl\{ x + iy : |x| < \delta, |y-\phi| < \delta \bigr\}$, is open in $\CC$. Now, $exp \bigl[ i (-\delta, \delta) \bigr] = E \cap S^{1}$, and hence is open in $S^{1}$. In other words $\bigl\{ (\cos y, \sin y) : |y-\phi| < \delta \bigr\}$ is open in $S^{1}$. We conclude
    \begin{equation}  \label{E:wrap.is.open}
      \text{If } G_{1}, \ldots, G_{n} \text{ are open in $\RR$ then } 
        \wrap( G_{1} \times \cdots \times G_{n} ) \text{ is open in } \D . 
    \end{equation}
(Or one could use a similar argument employing the Inverse Function Theorem, Boothby \cite[Theorem (6.4), pp.\ 42--43]{wmB75}.) 

  \begin{remark}[Covering space]  \label{R:covering.space.2}
We have
    \begin{equation}  \label{E:Rn.is.cov.space.of.D.2}
      (\RR^{n}, \wrap \,) \text{  is a covering space of } \D . 
    \end{equation}
(Massey \cite[pp.\ 145--146]{wsM67.Massey}). To prove this let 
$x = (x_{1}, \ldots, x_{n}) \in \D$ be arbitrary. Thus, $x_{i} \in S^{1}$. Let 
    \begin{equation}  \label{E:elementary.nbhd.for.D.2}
      \clU := \clU_{x} := \prod_{i=1}^{n} \bigl( S^{1} \setminus \{-x_{i}\} \bigr) 
        \subset \D .
    \end{equation}
Clearly, $\clU$ is open and arcwise and simply connected. Moreover, $\clU$ contains the open ball $\mcl{B}_{\pi}(x)$ centered at $x$ with radius $\pi$. (See \eqref{E:geod.dist.on.D.2}.) In fact, 
$\clU$ is just the radius $\pi$ open ball, $\X_{\pi}(x)$, centered at $x$ w.r.t.\ the metric $\chi$ defined in remark \ref{R:another.metric}. 

In this remark $i$ will denote $\sqrt{-1}$. Any point $x' \in \D$ is equivalent to an $n$-tuple \linebreak
$\bigl( \exp(i \phi'_{1}), \ldots, \exp(i \phi'_{n}) \bigr)$ for some 
$\phi'_{1}, \ldots, \phi'_{n} \in \RR$. 
Write $x = \bigl( \exp(i \phi_{1}), \ldots, \exp(i \phi_{n}) \bigr)$ and 
$\blds{\phi} := (\phi_{1}, \ldots, \phi_{n})$. 
Given $j \in \NN_{n}$, consider the set 
$C_{j} := \CC \setminus \bigl\{ - r \exp(i \phi_{j}) : r \in [0, \infty) \bigr\}$, 
where $\CC$ is the complex plane which we think of as $\RR^{2}$. $C_{j}$ is an open, simply connected region in $\CC$ and $\exp(i \phi_{1}) \in C_{j}$. Therefore, by Rudin \cite[Theorem 13.18(g), p.\ 263]{wR66.realcmplx}, there exists a holomorphic function $\log_{j} : C_{j} \to \CC$ s.t.\ $\exp \circ \log_{j}$ is the identity 
on $C_{j}$. Let $\arg_{\phi_{j}} = \Im \circ \log_{j} : C_{j} \to \RR$, 
where $\Im(z) \in \RR$ is the imaginary part of $z \in \CC$. 
So $\arg_{\phi_{j}}$ is smooth. 
Let $S^{1}_{j} := S^{1} \cap C_{j} = \bigl\{ z \in C_{j} : |z| = 1 \bigr\}$. 
Obviously, $\arg_{\phi_{j}}(S^{1}_{j})$ is an open interval $I_{j}$ of length $2 \pi$. We may arrange things so that the midpoint of the interval is $\phi_{j}$, i.e.,  
$I_{j} = \phi_{j} + (-\pi, \pi)$. For $\phi \in I_{j}$, define 
    \begin{equation}  \label{E:arg.defn}
      \arg_{\phi_{j}} \bigl[ (\cos \phi, \sin \phi) \bigr] 
        = \arg_{\phi_{j}} \bigl[ \exp(i \phi) \bigr] = \phi . \; 
          \arg_{\phi_{j}} : C_{\phi_{j}} \to \RR \text{ is an isometry.}
    \end{equation}
(See also \eqref{E:arg!.maties!} and \eqref{E:arg!.maties!2}.) An alternative notation is as follows. Let $z = (\cos \phi, \sqrt{-1} \sin \phi)$. Then there is a smooth, Lipschitz version of $\arg$ defined on $S^{1} \setminus \{-z\}$. Denote that version by 
    \begin{equation}  \label{E:arg.superscript}
      \arg^{z} : S^{1} \setminus \{-z\} \to \RR \text{ with } \arg^{z}(z) \in (-\pi, \pi] .
    \end{equation}
$\arg^{z} \bigl( S^{1} \setminus \{-z\} \bigl)$ is an open interval of length $2 \pi$ 
with $\arg^{z}(z)$ as midpoint. 

Let $\mbf{V} := I_{1} \times \cdots \times I_{n}$. Recall that if $\mbf{w} = (w_{1}, \ldots, w_{k}) \in \RR^{k}$ for some $k = 1, 2, \ldots$, then its $L^{\infty}$-norm is $|\mbf{w}|_{\infty} = \max_{i} |w_{i}|$. Then $\mbf{V}$ is just the radius $\pi$ open $L^{\infty}$ ball, $\mbf{B}^{\infty}_{\pi}(\blds{\phi}) \subset \RR^{n}$ about $\blds{\phi}$. The restriction $\wrap \, \restriction_{\mbf{V}}$ is the inverse of 
    \begin{multline}  \label{E:arg,matey!.2}
      \blds{\arg}_{\blds{\phi}} := \arg_{\phi_{1}} \times \cdots \times \arg_{\phi_{n}} 
      \text{ is defined and smooth on } C_{1} \times \cdot \times C_{n} .  \\
         \text{ We may assume } \blds{\arg}_{\blds{\phi}}(x) = \blds{\phi} .
    \end{multline}
 
The components of $\wrap^{-1}(\clU)$ are precisely the sets 
$(I_{1} \times \cdots \times I_{n}) + 2 \pi \mbf{k}$ for $\mbf{k} \in \ZZ^{n}$. 
The sets $\clU_{x}$, $x \in \D$, are the ``elementary neighborhoods'' of the covering. By \eqref{E:geod.dist.on.D.2}, 
    \begin{equation}  \label{E:arg.is.isom}
      \blds{\arg}_{\blds{\phi}} : \clU_{x} \to I_{1} \times \cdots \times I_{n} \text{ and }
        \wrap \text{ in the other direction are isometries.}
    \end{equation}
  \end{remark}

Here is a rather obvious fact we will use more than once. 
    \begin{multline} \label{E:modulo}
      \text{If } \lambda > 0 \text{ then for every } \omega \in \RR 
        \text{ there exists a unique } \gamma \in (-\lambda, \lambda] \\
         \text{ s.t.\ } \omega - \gamma \text{ is an integral multiple of } 2 \lambda .
    \end{multline}
To see this, let $\lfloor \cdot \rfloor$ be the integer part function that takes 
$s \in \RR$ to the largest integer no bigger than $s$. 
Let $\zeta := \omega - \lfloor \omega/(2 \lambda) \rfloor \times 2 \lambda
 \in [0, 2 \lambda)$. 
Define $f : [0, 2 \lambda) \to (-\lambda, \lambda]$ as follows.
	\begin{equation*}  
		f(\alpha) :=
			\begin{cases}
				\alpha, &\text{ if } 0 \leq \alpha \leq \lambda, \\
				\alpha - 2 \lambda, &\text{ if } \lambda < \alpha < 2 \lambda.
			\end{cases}
	\end{equation*} 
Then $\alpha - f(\alpha) = $ 0 or $2 \lambda$. Finally, let $\gamma = f(\zeta)$. Then, considering the cases $\zeta \in [0, \lambda]$ and $\zeta \in (\lambda, 2 \lambda)$ separately we see that $\gamma \in (-\lambda, \lambda]$ and
    \begin{multline*}
      \omega - \gamma = (\omega - \zeta) + \bigl( \zeta - f(\zeta) \bigr)
        = \bigl( \omega - \omega + \lfloor \omega/(2 \lambda) \rfloor 
          \times 2 \lambda \bigr) + \bigl( \zeta - f(\zeta) \bigr) \\
             \in \bigl( \lfloor \omega/(2 \lambda) \rfloor + \{ 0, 1\} \bigr) 
               \times 2 \lambda .
    \end{multline*}
The set in parentheses at the end consists of integers. This proves \eqref{E:modulo}.
Thus, \eqref{E:wrap.defn.2} with $\lambda = \pi$ implies we may take
$\phi_{i} \in ( -\pi, \pi ], \, i \in \NN_{n}$.

Let $\blds{\theta} = ( \theta_{1}, \ldots, \theta_{n} )$. Define
	\begin{multline} \label{E:psi.defn.2}
		\psi_{\blds{\phi}} : ( \theta_{1}, \ldots, \theta_{n} )  \\
			\mapsto \bigl( \cos (\phi_{1}+\theta_{1}), \sin (\phi_{1}+\theta_{1}), 
			  \cos (\phi_{2}+\theta_{2}), \sin (\phi_{2}+\theta_{2}), \\
			     \ldots, \cos (\phi_{n}+\theta_{n}), \sin (\phi_{n}+\theta_{n}) \bigr)
			       = \wrap(\blds{\phi} + \blds{\theta}) \in \D,  \\ 
				   \theta_{i} \in ( -\pi, \pi], \, i \in \NN_{n}.
	\end{multline}
Then $\psi_{\blds{\phi}}$ parametrizes a coordinate neighborhood of $x$. 
 
Put on $\D$ the Riemannian metric induced by inclusion 
$\D \hookrightarrow \RR^{2n}$. We take another look at the Riemannian metric and geodesics on $\D$. (We already did this. See \eqref{E:geod.on.prod.of.spheres}.) Let $\blds{\phi} \in \RR^{n}$ and let $x = \wrap(\blds{\phi})$. For $i \in \NN_{n}$, let
	\begin{equation}   \label{E:zi.defn}
		\mbf{z}_{i} := \mbf{z}_{i}(\blds{\phi}) := \mbf{z}_{i}[x]
		    := (0, \ldots, 0, - \sin \phi_{i}, \cos \phi_{i}, 0, \ldots, 0) \in \RR^{2n},
	\end{equation}
 where a sine is in position $2i-1$ and a cosine is in position $2i$. Thus, 
 $\mbf{z}_{i}(\blds{\phi}) = \tfrac{\partial}{\partial \phi_{i}} \wrap(\blds{\phi})$, 
 only depends on $\wrap(\blds{\phi})$, and
	 \begin{equation} \label{E:z's.are.orthon}
		\mbf{z}_{1}, \ldots, \mbf{z}_{n} \text{ are orthonormal.}
	\end{equation}
 Note that 
     \begin{multline}  \label{E:x.mapsto.z.Lip}
        (- \sin \phi_{i}, \cos \phi_{i}) = 
          ( \cos \phi_{i}, \sin \phi_{i} )
            \begin{pmatrix}
                 0 & 1 \\
                 -1 & 0
            \end{pmatrix} . \\
            \text{Hence, the map } \wrap(\blds{\phi}) \mapsto \mbf{z}_{i}(\blds{\phi}) \text{ is well-defined and Lipschitz}.
    \end{multline} 
 
Tangent vectors to $\D$ have the form $(x, \blds{\phi})$, where $x \in \D$ and 
$\blds{\phi} \in \RR^{n}$. (Sometimes instead of writing tangent vectors like 
$v_{x} \in T_{x} \D$ or $(x, v)  \in T_{x} \D$ we might just write $v$.)  Alternatively, a tangent vector at $x \in D$ is a differential operator acting on smooth $\RR$-valued functions defined in a neighborhood of $x$ (Boothby \cite[Definition (1.1), p.\ 106]{wmB75}). Using the notation of Boothby \cite[Theorem (1.2), p.\ 107]{wmB75} and recalling \eqref{E:psi.defn.2}, we can interpret $\mbf{z}_{i}$ as 
$\wrap_{\ast}(\partial/\partial \phi_{i})$. So if $G$ is a smooth 
$\RR$-valued functions defined in a neighborhood of $x = \wrap(\blds{\phi})$ we have $\mbf{z}_{i}(G) = (\partial/\partial \phi'_{i}) \restriction_{\blds{\phi}'=\blds{\phi}} 
\bigl[ G \circ \wrap(\blds{\phi}') \bigr]$.

$S^{1}$ can be thought of as a subset of $\RR^{2}$. Therefore, $\D$ can be thought of as a subset of $\RR^{2n}$.
Let $inc : \D \hookrightarrow \RR^{2n}$. Any smooth function, $F$, defined on an open set in $\D$ can be extended, via the Tubular Neighborhood Theorem \ref{P:tubular.nbhd.thm}, to an open subset of $\RR^{2n}$. We will always assume such functions are so extended. We will use $\mbf{z}_{i}$ to denote the vector $inc_{\ast}(\mbf{z}_{i}) \in T_{\wrap(\blds{\phi})} \RR^{2n}$. Thus, by Boothby \cite[Theorem  (1.6), p.\ 109]{wmB75},  
With that understanding we have,
	\begin{equation} \label{E:zi.vec.defn.2}
		\mbf{z}_{i} := \mbf{z}_{i,x} := \mbf{z}_{i}(\blds{\phi}) 
		  = - \sin \phi_{i} \frac{\partial}{\partial y_{2i-1}} 
		    + \cos \phi_{i} \frac{\partial}{\partial y_{2i}} 
		      \in T_{\wrap(\blds{\phi})} \RR^{2n}.
	\end{equation}
It follows from \eqref{E:zi.vec.defn.2}, that $\mbf{z}_{1,x}, \ldots, \mbf{z}_{n,x}$ are orthonormal w.r.t.\ the Riemannian metric in the ambient space $\RR^{2n}$ and, hence, in the $n$-dimensional space $T_{x} \D$. Thus, 
	\begin{multline}  \label{E:z's.span.TxD.2}
		\mbf{z}_{1}(\blds{\phi}), \ldots, \mbf{z}_{n}(\blds{\phi}) 
		\text{ are orthonormal w.r.t.} \\
		  \text{ the Riemannian metric on } \D 
		    \text{ and span } T_{\wrap \, (\blds{\phi})} \D.
	\end{multline}
Since $\mbf{z}_{1}, \ldots, \mbf{z}_{n}$ are orthonormal, it follows that, 
relative to $\mbf{z}_{1}, \ldots, \mbf{z}_{n}$, the matrix of the Riemannian metric on $\D$ is the identity matrix $I_{n}$. Therefore, by Boothby \cite[Corollary (3.8), p.\ 318]{wmB75}, the Christoffel symbols of the second kind are all 0:
    \begin{equation}  \label{E:Christ.awful.symbols}
      \Gamma_{ij}^{k} = 0, \quad i,j,k = 1, \ldots, n . 
    \end{equation}
(Hence, by Boothby \cite[Corollary (4.3), p.\ 323]{wmB75}, it follows that $\D$ has the same curvature tensor as $\RR^{n}$, \emph{viz.} 0.)

For $i \in \NN_{n}$, 
	\begin{multline}   \label{E:ei.in.Rn}
		\text{Let } e_{i} \text{ be the coordinate vector } 
		  (0, \ldots, 0, 1, 0, \ldots, 0) \in \RR^{n} \\
		  \text{ with the ``1'' in the $i^{th}$ position.}
	\end{multline} 
Recall the definition, \eqref{E:1n.col.vec.defn}, of $1^{n}$. Strictly speaking, $1^{n}$ is a column vector, but in this appendix, instead of using the more correct $1_{n}$, we might sometimes treat $1^{n}$ as a row vector. (Sorry.) We have,
$1_{n} = e_{1} + \cdots + e_{n}$. 

Recall the function $\wrap \,$ defined in \eqref{E:wrap.defn.2}. 
Let $\wrap_{\ast} : T_{\blds{\phi}} \RR^{n} \to T_{\wrap(\blds{\phi})}$ denote its differential. We may identify $\RR^{n}$ with $T_{\blds{\phi}} \RR^{n}$, of course, and under that identification $e_{i}$ is identified with $\partial/\partial \phi_{i}$. Currently, 
$\mbf{z}_{i}$ is defined to be 
$inc_{\ast} \circ \wrap_{\ast}(\partial/\partial \phi_{i})$
For convenience we identify $\mbf{z}_{i}$ with 
$\wrap_{\ast}(\partial/\partial \phi_{i})$ so
    \begin{multline}  \label{E:wrap.differential.2}
         \text{At } \blds{\phi}, \;  \wrap_{\ast} : 
           e_{i} \mapsto \mbf{z}_{i}(\blds{\phi}) \; (i \in \NN_{n}). \\
             \text{ Therefore, } \wrap_{\ast} \text{ is a local isometry of } \RR^{n} 
               \text{ onto } T_{\wrap(\blds{\phi})} \D .
    \end{multline}

In the problem we consider here, the directional location problem, the group $G$ is just the symmetric group, $S_{n}$, on $n$ symbols. By definition,  
    \begin{equation*} \label{E:position.of.i.in.g(1,...,n).2}
      \text{In } g(1, 2, \ldots, n) \text{ the symbol $i$ is in position } g(i) .
    \end{equation*}

$G$ acts on $\RR^{n}$ and $\RR^{2n}$ as follows. 
	\begin{equation}  \label{E:perm.action}
	    g(\phi_{1}, \ldots, \phi_{n}) := ( \ldots, \phi_{g(j)}, \ldots ) \text{ and } 
		g(x_{1}, \ldots, x_{2n}) := \bigl( \ldots, x_{2g(j)-1}, x_{2g(j)}, \ldots \bigr) .
	\end{equation}
Here, $\phi_{g(j)}$ is the $j^{th}$ coordinate of $g(\blds{\phi})$ and, 
in $g(x_{1}, \ldots, x_{2n})$, the pair $x_{2g(j)-1}, x_{2g(j)}$ occupy positions 
$2j-1$, and $2j$, resp. Notice that in these actions 
    \begin{equation}  \label{E:g.is.linear}
      g \text{ is linear. }
    \end{equation}
Let $x = (x_{1}, \ldots, x_{2n})$ and let $g(x)_{\ell}$ denote the $\ell^{th}$ entry in $g(x)$, $\ell = 1, \ldots, 2n$ so 
$\{ g(x)_{\ell} : \ell \text{ is odd} \} = \{ x_{\ell} : \ell \text{ is odd} \}$ and the same for even indices. By \eqref{E:wrap.defn.2}, 
    \begin{equation}  \label{E:g.wrap.commute.2}
       g \bigl[ \wrap(\blds{\phi}) \bigr] = \wrap \bigl[ g(\blds{\phi}) \bigr] .
    \end{equation} 

Let $m = 1, \ldots, 2n$ and $g \in G$. We carefully work through Boothby \cite[Theorem (1.2), p.\ 107]{wmB75} to compute $g_{\ast}$. Let $\blds{\phi} = (\phi_{1}, \ldots, \phi_{n}) \in \RR^{n}$. 
Let $x = \wrap(\blds{\phi})$ and let $F$ be a smooth function mapping a neighborhood of $x$ in $\D$ into $\RR$. We may assumed that $F$ is actually defined smoothly in a neighborhood of $x$ in $\RR^{2n}$. Let $m = 1, \ldots, 2n$.
Then 
    \begin{multline}  \label{E:g.star.partial.y}
        g_{\ast} \left( \frac{\partial}{\partial y_{m}} \restriction_{y = x} \right) (F) 
          = \left( \frac{\partial}{\partial y_{m}} \restriction_{y = x} \right) (F \circ g) \\
            = \sum_{\ell=1}^{2n} \frac{\partial F(w)}{\partial w_{\ell}} \restriction_{w = g(x)} \;
              \frac{\partial g(y)_{\ell}}{\partial y_{m}} \restriction_{y=x} 
                = \frac{\partial F(w)}{\partial w_{\ell}} \restriction_{w = g(x)} ,
    \end{multline}
where $g(y)_{\ell} = y_{m}$. 
Which $\ell$ solves $g(y)_{\ell} = y_{m}$ for generic $y \in \RR^{2n}$? Suppose $m$ is odd, say $m = 2k-1$, then the $\ell$ must be odd. So for some $i \in \NN_{n}$ we have $\ell = 2i-1$ and $2 g(i) - 1= m = 2k-1$. 
Thus, $i = g^{-1}(k)$ and 
    \begin{equation*}
        g_{\ast} \left( \frac{\partial}{\partial y_{2k-1}} \restriction_{y = x} \right) (F) 
           = \frac{\partial F (w)}{\partial w_{2 g^{-1}(k) - 1}} \restriction_{w = g(x)} .
    \end{equation*}
Similarly, 
    \begin{equation*}
        g_{\ast} \left( \frac{\partial}{\partial y_{2k}} \restriction_{y = x} \right) (F) 
           = \frac{\partial F(w)}{\partial w_{2 g^{-1}(k)}} \restriction_{w = g(x)} .
    \end{equation*} 

If $k \in \NN_{n}$, then the coordinate of $g(\blds{\phi})$ in which one would find $\phi_{k}$ is $g_{-1}(k)$. Apply the preceding with $m = 2k-1$ or $m = 2k$. Then, from \eqref{E:zi.vec.defn.2},
    \begin{align*}
      g_{\ast} \bigl[ \mbf{z}_{k}(\blds{\phi}) \bigr] (F)
        &= - \sin(\phi_{k}) \left( \frac{\partial F(w)}
          {\partial w_{2 g^{-1}(k) - 1}} \restriction_{w = g(x)} \right)
            + \cos(\phi_{k}) \left( \frac{\partial F(w)}
              {\partial w_{2 g^{-1}(k)}} \restriction_{w = g(x)} \right) \\
        &= - \sin \bigl[ g(\blds{\phi})_{g^{-1}(k)} \bigr] 
           \left( \frac{\partial F(w)}{\partial w_{2 g^{-1}(k) - 1}} 
             \restriction_{w = g(x)} \right)
              + \cos \bigl[ g(\blds{\phi})_{g^{-1}(k)} \bigr] 
              \left( \frac{\partial F(w)}{\partial w_{2 g^{-1}(k)}} 
                \restriction_{w = g(x)} \right) \\
        &= \mbf{z}_{g^{-1}(k)} \bigl[ g(\blds{\phi}) \bigr] (F) 
            \in T_{\wrap [ g(\blds{\phi}) ]} \D .
    \end{align*} 
We conclude
	\begin{equation}  \label{E:perms.act.on.z}
		g_{\ast} \bigl[ \mbf{z}_{i}(\blds{\phi}) \bigr] 
		= \mbf{z}_{g^{-1}(i)} \bigl[ g(\blds{\phi}) \bigr] ,
		    \quad i \in \NN_{n}.
	\end{equation}
Thus, by \eqref{E:z's.span.TxD.2}, $g_{\ast} \mbf{z}_{i}$ 
($i \in \NN_{n}$) are still orthonormal. Hence, 
	\begin{equation}  \label{E:Riem.metric.is.G.invar}
		\text{The Riemannian metric on } \D \text{ is } G-\text{invariant}.
	\end{equation}
 
Next, we work out the form of geodesics in $\D$. (Yes, we examined that issue before: \eqref{E:geod.on.prod.of.spheres}) Let $\blds{\phi} \in \RR^{n}$ 
and let $x = \wrap(\blds{\phi})$. Let $I \subset \RR$ be open with $0 \in I$.
Let $\blds{\theta} = ( \theta_{1}, \ldots, \theta_{n} ) : I \to \RR^{n}$ be differentiable. 
Recall the definition, \eqref{E:psi.defn.2}, of $\psi_{\blds{\phi}}$. Consider the curve
	 \begin{equation} \label{E:varkappa.defn}
		\varkappa(t) := \varkappa_{\blds{\theta}}(t) 
		  := \varkappa_{\blds{\theta}, \blds{\phi}}(t)
		    := \psi_{\blds{\phi}} \bigl( \theta_{1}(t), \ldots, \theta_{n}(t) \bigr) 
		      = \wrap \bigl[ \blds{\phi} + \blds{\theta}(t) \bigr] \in (S^{1})^{n} , 
		            \; t \in I.
	\end{equation}
(``$\varkappa$'' is a variant of ``$\kappa$'', ``kappa''.) Thus, if 
$\blds{\theta}(0) = (0, \ldots, 0)$, we have 
$\varkappa(0) = x := \wrap(\blds{\phi})$. Using a dot to denote differentiation w.r.t.\ $t$, by \eqref{E:wrap.differential.2}, \eqref{E:zi.vec.defn.2}, and 
Boothby \cite[Theorem (1.6), p.\ 109]{wmB75}, we have
  \begin{multline}  \label{E:Gamma.dot.formla}
   \dot{\varkappa}_{\blds{\theta}}(t)
     = \sum_{i=1}^{n} \dot{\theta}_{i}(t) \Bigl( - \sin \bigl[ \phi_{i} 
	+ \theta_{i}(t) \bigr] \tfrac{\partial}
	  {\partial y_{2i-1}} \restriction_{y = \varkappa(t)} 
	  + \cos \bigl[ \phi_{i} + \theta_{i}(t) \bigr] \tfrac{\partial}
	    {\partial y_{2i}} \restriction_{y = \varkappa(t)} \Bigr) \\
	      = \sum_{i=1}^{n} \dot{\theta}_{i}(t) \mbf{z}_{i, \varkappa(t)}
		\in T_{\varkappa_{\blds{\theta}}(t)} \D.
  \end{multline}

By Boothby \cite[Definition (5.1), p.\ 326]{wmB75}, in order for $\varkappa$ to be a geodesic, we must have $(D/dt)(dp/dt)$, where ``$(D/dt)$'' denotes covariant differentiation on $\D$. Now, by \eqref{E:Christ.awful.symbols} and Boothby \cite[3.11, p.\ 319]{wmB75}, we have
    \begin{equation*}
      \frac{D}{dt} \dot{\varkappa}_{\blds{\theta}}(t)
      = \sum_{k=1}^{n} \left(\ddot{\theta}_{i}(t) \mbf{z}_{i,  \varkappa(t)} 
        + \sum_{i,j=1}^{n} \Gamma_{ij}^{k} \dot{\theta}_{i}(t)^{2} \right) 
      = \sum_{k=1}^{n} \ddot{\theta}_{i}(t) \mbf{z}_{i,  \varkappa(t)} .
    \end{equation*} 
Hence, by Boothby \cite[Definition (5.1), p.\ 326]{wmB75},
	\begin{equation*}
		\varkappa_{\blds{\theta}} \text{ is a geodesic in } \D 
		  \text{ if and only if } \ddot{\theta}_{i}(t) = 0, 
			\quad t \in I, \,i \in \NN_{n}
	\end{equation*}
 I.e., for some constants $a_{1}, b_{1}, \ldots, a_{n}, b_{n}$, 
	\begin{equation*} 
		\varkappa_{\blds{\theta}} \text{ is a geodesic in } \D 
		\text{ if and only if } 
		  \theta_{i}(t) = a_{i} + b_{i} t, \quad t \in I, \,i \in \NN_{n}.
	\end{equation*}
$\varkappa_{\blds{\theta}}(0)= x = \wrap(\blds{\phi})$ if and only if $a_{1}, \ldots, a_{n}$ are each integral multiples of $2 \pi$. But in general, let $\mbf{a} = (a_{1}, \ldots, b_{n})$ 
and $\bb = (b_{1}, \ldots, b_{n}) \in \RR^{n}$ and define $\mbf{a} + \bb \times$ to be 
the function mapping $\RR$ to $\RR^{n}$ 
defined by $(\mbf{a} + \bb \times) (t) = \mbf{a} + t \bb$. Thus, 
	\begin{multline}  \label{E:when.Gamma.is.geod.2}
		\text{A curve } \gamma : \RR \to \D 
		  \text{ is a geodesic passing through } 
		  x =  \wrap(\mbf{a})  \\
		    \text{ if and only if } \gamma = 
		      \varkappa_{\mbf{a} + \bb \times} 
		      \text{ for some } \mbf{a}, \bb \in \RR^{n}.
	\end{multline}
This just reproves \eqref{E:geod.on.prod.of.spheres} in the case $\nvar = 1$.

Moreover, with $\mbf{a}, \bb \in \RR^{n}$ as before then, by \eqref{E:Gamma.dot.formla},  
	\begin{equation}  \label{E:geod.deriv.2}
		\dot{\varkappa}_{\mbf{a} + \bb \times}(t)
		  = \sum_{i=1}^{n} b_{i} \mbf{z}_{i, \varkappa(t)}
 	  \end{equation}
	
Let $t_{1}, t_{2} \in I$. Then, by Boothby \cite[pp.\ 185--186]{wmB75},
\eqref{E:z's.are.orthon}, and \eqref{E:geod.deriv.2},
	\begin{multline}  \label{E:geod.length}
		\text{Length of geodesic arc } 
		  \varkappa_{\mbf{a} + \bb \times}(s), \text{ with } 
		  s \in [t_{1}, t_{2}], \\
		    \text{ is } \int_{t_{1}}^{t_{2}} \left| \sum_{i=1}^{n} b_{i} 
		      \mbf{z}_{i, \varkappa(t)} \right| \, dt 
		        = |t_{2} - t_{1}| | \bb | .  
	\end{multline}
Thus, if $\blds{\beta}_{i} = (\beta_{i1}, \ldots, \beta_{in}) \in \RR^{n}$ and 
$x_{i} = \wrap(\blds{\beta}_{i})$ ($i=1,2$) then 
$\varkappa_{\blds{\beta}_{1} + (\blds{\beta}_{2} - \blds{\beta}_{1}) \times}$ restricted to the interval $[0,1]$ is a geodesic arc joining $x_{1}$ and $x_{2}$. By \eqref{E:geod.length}, 
its length is $|\blds{\beta}_{2} - \blds{\beta}_{1}|$.
By \eqref{E:modulo}, by adding integral multiples of $2 \pi$ to $\beta_{1j}$ 
or $\beta_{2j}$ we do not change $x_{1}$ or $x_{2}$, but we can arrange 
$| \beta_{2j} - \beta_{1j} | \leq \pi$ ($j \in \NN_{n}$). Hence, by  \eqref{E:geod.length}, 
    \begin{multline} \label{E:geod.dist.on.D.2}
      \text{The geodesic distance, } \rho(x_{1}, x_{2}), \text{ between } 
      x_{1} = \wrap(\blds{\beta}_{1}) 
      \text{ and } x_{2} = \wrap(\blds{\beta}_{2}) \\
        \text{ is no greater than } |\blds{\beta}_{2} - \blds{\beta}_{1}|, \\
      \text{ with equality if and only if } 
      | \beta_{2j} - \beta_{1j} | \leq \pi \quad (j \in \NN_{n}).   
   \end{multline}   
In particular, $\wrap$ is a local isometry, hence Lipschitz.

Recall \eqref{E:angle.between.vectors}. $\angle$ is clearly a metric on $S^{1}$.
Topologizing $S^{1}$ using $\angle$, we see that 
        \begin{equation}  \label{E:D.has.prod.top}
          \D \text{ has the product topology. }
        \end{equation}

Recall \eqref{E:angle.between.vectors}. 
If $x = (x_{1}, \ldots, x_{n}), y = (y_{1}, \ldots, y_{n}) \in \D$ define 
    \begin{equation}  \label{E:angle.vector}
      \angle(x,y) := \bigl( \angle(x_{1}, y_{1}), \ldots, \angle(x_{1}, y_{1}) \bigr). 
    \end{equation}
Then, by \eqref{E:geod.dist.on.D.2}, $\rho(x,y) = \bigl| \angle(x,y) \bigr|$, the Euclidean norm of $\angle(x,y)$. Recall that if $w = (w_{1}, \ldots, w_{n}) \in \RR^{n}$ for some 
$n = 1, 2, \ldots$, then its $L^{\infty}$-norm is $|w|_{\infty} = \max_{i} |w_{i}|$. We have
    \begin{equation}  \label{E:L.infty.L.2.norms.3}
      n^{-1/2} |w| \leq |w|_{\infty} \leq |w| ,
    \end{equation}

Specifically, let 
$\blds{\xi} = (\xi_{1}, \ldots, \xi_{n}), \blds{\zeta} = (\zeta_{1}, \ldots, \zeta_{n}) 
\in (-\pi, \pi]^{n}$. 
To construct a shortest geodesic, 
$\varkappa_{\mbf{a} + \bb \times} : [0,1] \to \D$, 
from $\wrap(\blds{\xi})$ to $\wrap(\blds{\zeta})$ we must, for each $i = 1, \dots, n$, choose $a_{i}$ and $b_{i}$. 
If $|\zeta_{i} - \xi_{i}| \leq \pi$, let $a_{i} = \xi_{i}$ and $b_{i} = \zeta_{i} - \xi_{i}$. Suppose $\xi_{i} < 0 < \xi_{i}+\pi < \zeta_{i}$. 
Then let $a_{i} = \xi_{i} + 2 \pi$ and $b_{i} = \zeta_{i} - (\xi_{i}+2 \pi)$. Similarly if
$\zeta_{i} < 0 < \zeta_{i}+\pi < \xi_{i}$.

Examine this another way. Suppose $\blds{\xi}, \blds{\zeta} \in (-\pi, \pi]^{n}$ as before. By \eqref{E:when.Gamma.is.geod.2}, any geodesic connecting $\wrap(\blds{\xi})$ to $\wrap(\blds{\zeta})$ 
has the form $\varkappa_{\mbf{a} + \bb \times}$. We restrict ourselves to geodesics on $[0,1]$. 
We must have $\wrap(\mbf{a}) = \wrap(\blds{\xi})$ and
$\wrap(\mbf{a} + \bb) = \wrap(\blds{\zeta})$. 
Therefore, $\mbf{a} = \blds{\xi} + 2 \mbf{k}_{1} \pi$ 
and $\bb = (\blds{\zeta} - 2 \mbf{k}_{2} \pi) - (\blds{\xi} + 2 \mbf{k}_{1} \pi) 
 = \blds{\zeta} - \bigl( \blds{\xi} + 2 (\mbf{k}_{1} + \mbf{k}_{2}) \pi \bigr)$ for some 
 $\mbf{k}_{1}, \mbf{k}_{2} \in \ZZ^{n}$. 
These are the only geodesics on $[0,1]$ joining $\wrap(\blds{\xi})$ 
to $\wrap(\blds{\zeta})$, but not all of them are shortest.

Now, suppose $\blds{\xi}, \blds{\zeta} \in (-\pi, \pi]^{n}$. Therefore, for each $i$, we have $|\zeta_{i} - \xi_{i}| < 2 \pi$. 
Let $\mbf{k} \in \ZZ^{n}$. And suppose for each $i$ we have 
$\bigl| \zeta_{i} - \xi_{i} - 2 k_{i} \pi  \bigr| \leq \pi$. 
Thus, 
    \begin{equation*}
      \pi \geq \bigl| \zeta_{i} - \xi_{i} - 2 k_{i} \pi  \bigr| \geq 2 |k_{i}| \pi  - |\zeta_{i} - \xi_{i}| 
        > 2 \bigl( |k_{i}| - 1 \bigr) \pi .
    \end{equation*}
But by \eqref{E:modulo}, there is \emph{some} $k_{i} \in \ZZ$ s.t.\ 
$\bigl| \zeta_{i} - \xi_{i} - 2 k_{i} \pi  \bigr| \leq \pi$. In summary, 
    \begin{align}  \label{E:optimum.k=-1,0,1}
      &\text{If } \blds{\xi}, \blds{\zeta} \in (-\pi, \pi]^{n} \text{ and }
       \bigl| \zeta_{i} - \xi_{i} - 2 k_{i} \pi  \bigr| \leq \pi \text{ for } i \in \NN_{n} 
         \text{ then } \mbf{k} \in \{ -1, 0, 1 \}^{n}. \notag \\
      &\text{Conversely, if } \blds{\xi}, \blds{\zeta} \in (-\pi, \pi]^{n}
        \text{ then for some } \mbf{k} \in \{ -1, 0, 1 \}^{n} \\
      & \qquad \text{ we have }
          -\pi < \zeta_{i} - \xi_{i} - 2 k_{i} \pi \leq \pi \text{ for } i \in \NN_{n} . \notag
    \end{align} 

Let $x = \wrap(\blds{\phi}) \in \D$ be as in \eqref{E:wrap.defn.2}. 
Let $Exp_{x} : T_{x} \D \to \D$ be the exponential map. 
If $\bb := ( b_{1}, \ldots, b_{n} )$ is constant, then, by \eqref{E:when.Gamma.is.geod.2} with $t = 1$, \eqref{E:geod.deriv.2}, and \eqref{E:varkappa.defn}, 
	\begin{multline}  \label{E:Exp.psi.2}
	     Exp_{x} (b_{1} \mbf{z}_{1,x} + \cdots + b_{n} \mbf{z}_{n,x})  = 
	       Exp_{x} \bigl( b_{1} \mbf{z}_{1}(\blds{\phi}) + 
	          \cdots + b_{n} \mbf{z}_{n}(\blds{\phi}) \bigr) \\
	           = \varkappa_{\mbf{0} + \bb \times, \blds{\phi}} (1)
		     = \wrap(\blds{\phi} + \bb).
	\end{multline}
Let $\blds{\theta} = (\theta_{1} \ldots, \theta_{n}), \blds{\phi} \in \RR^{n}$, 
$\mbf{w} = \sum_{i=1}^{n} \theta \, \mbf{z}_{i} (\blds{\phi}) 
\in T_{\wrap(\blds{\phi})} \D$, 
and $x = \wrap(\blds{\phi})$. We can give a short proof of \eqref{E:gExp=Expg.star} in this case. By \eqref{E:Exp.psi.2}, \eqref{E:g.wrap.commute.2}, \eqref{E:g.is.linear}, and \eqref{E:perms.act.on.z} we have the following. 
    \begin{align} \label{E:g.circ.Exp.2}
        g \circ Exp_{x}(\mbf{w}) 
          &= (g \circ \wrap \, )(\blds{\phi} + \blds{\theta}) 
            = (\wrap \circ g) (\blds{\phi} + \blds{\theta}) \notag \\
          &= \wrap \bigl[ g (\blds{\phi} )+ g(\blds{\theta}) \bigr] 
            = Exp_{(\wrap \, \circ g)(\blds{\phi})} 
                 \left( \sum_{i=1}^{n} \theta_{g(i)} \mbf{z}_{i}
                   \bigl[g(\blds{\phi})\bigr] \right) \notag \\
          &= Exp_{(g \circ \wrap \, )(\blds{\phi})} 
            \left( \sum_{i=1}^{n} \theta_{g(i)} \mbf{z}_{i}\bigl[g(\blds{\phi})\bigr] \right) \\
          &= Exp_{g(x)} \left( \sum_{i=1}^{n} \theta_{g(i)} 
            \mbf{z}_{i}\bigl[g(\blds{\phi})\bigr] \right) 
            \notag \\
          &= Exp_{g(x)} \left( \sum_{j=1}^{n} \theta_{j} 
            \mbf{z}_{g^{-1}(j)}\bigl[g(\blds{\phi})\bigr] \right) 
            \notag \\
          &= Exp_{g(x)} \bigl[ g_{\ast}(\mbf{w}) \bigr] \notag .
    \end{align}

Since, by \eqref{E:rho.is.top.metric}, we are using $\rho$ to denote the metric on $\D$, 
    \begin{equation}  \label{E:rho+.instead.of.xi+}
      \text{We shall write $\rho_{+}$ instead of $\xi_{+}$.} 
    \end{equation}
(See \eqref{E:xi.is.metric.on.D} and \eqref{E:xi+.from.2.metrics}.) Then lemma \ref{L:Exp.loc.Lip} becomes:
    \begin{lemma}[$Exp$ is locally Lipschitz]  \label{L:Exp.is.loc.Lip.in.circle.case}
      $Exp$ is locally Lipschitz on $T \D$ w.r.t.\ the metrics $\rho_{+}$ and $\rho$. Hence, by \eqref{E:local.Lip.is.Lip.on.compacts}, $Exp$ is Lipschitz on compact sets.
    \end{lemma}
This proves part of \textbf{property \ref{I:Exp.alpha.homeom} in definition \ref{D:fibering.by.cones}}. 

\section{Stratification of $\Pf_{1}$}  \label{S:stratify.P1}
Let $x,y \in S^{1}$. Define the \emph{(directed) angle from $x$ to $y$ in the positive direction} to be
    \begin{multline}  \label{E:angle.+}
      \angle_{+}(x,y) = 
        \text{minimum distance around $S^{1}$} \\
          \text{ in the counterclockwise direction from $x$ to } y. 
    \end{multline}
The definition of the \emph{(directed) angle, $\angle_{-}(x,y)$, from $x$ to $y$ in the negative direction} is same except the distance is measured in the clockwise direction. Thus, 
    \begin{multline*}
      \angle_{+}(x,y), \angle_{-}(x,y) \in [0, 2 \pi) , \;
        \angle_{+}(x,y) = 2 \pi - \angle_{-}(x,y), \\
          \text{ and } \angle_{+}(x,y) = \angle_{-}(y,x) ,
            \angle(x,y) = \min \bigl[ \angle_{+}(x,y), \angle_{-}(x,y) \bigr] .
    \end{multline*}
Recall \eqref{E:sign.function}. If $s \neq 0$, define 
    \begin{equation}  \label{E:s.directed.angle}
      \angle_{s}(x,y) := \angle_{sign(s)}(x,y) .
    \end{equation}

Let $s = \pm 1$ and $a \in (0, 2 \pi)$. Let 
$J_{s,a} := \bigl\{ (x,y) : \angle_{(-s)}(x,y) \leq a \bigr\}$.   
    \begin{equation}  \label{E:directed.angle.is.Lip}
      \angle_{s} \text{ is Lipschitz off } J_{s,a} .
    \end{equation} 

Recall, by \eqref{E:N.sub.n}, $\NN_{n} := \{1, \ldots, n \}$. Recall \eqref{E:Pk.defn.robust.loc.on.sphere}. 
For $i = 1, \ldots, \NN_{n}$, let
    \begin{multline}  \label{E:Pf.1i.defn}
      \Pf_{1i} := \bigl\{ \wrap(\beta) \in \D : 
        \beta = (\phi_{1}, \ldots, \phi_{1}, \phi_{2}, \phi_{1}, \ldots, \phi_{1}) 
          \in \RR^{n}, \\
            \text{ where $\phi_{2}$ is the $i^{th}$ coordinate} \bigr\}
    \end{multline}
$\Pf_{1i}$ is the $i^{th}$ \emph{lobe} of $\Pf_{1}$. Alternatively, 
    \begin{equation*}
      \Pf_{1i} = \{x \in \D : \text{ if $j,k \neq i$ then } x_{j} = x_{k} \} .
    \end{equation*}
    
Observe
    \begin{equation} \label{E:T.=.intersection.of.all.lobes}
      \text{If } i \neq j \text{ then } \Pf_{1i} \cap \Pf_{1j} = \T. \text{ Thus, }
        \T = \bigcap_{j=1}^{n} \Pf_{1j} .
    \end{equation}
It is obvious that $\Pf_{1j}$ is diffeomorphic to the torus in $\RR^{4}$. In particular,
    \begin{equation}  \label{E:Pf1=union.of.lobes}
        \Pf_{1j} \text{ is a compact two-dimensional differentiable manifold and }
          \Pf_{1} = \bigcup_{j=1}^{n} \Pf_{1j} .
    \end{equation}
By \eqref{E:T.=.intersection.of.all.lobes}, $\Pf_{1i}$ ($i \in \NN_{n}$) are not strata of $\Pf_{1}$ as defined in subsection \ref{SSS:conical.fibers}. (A stratification 
of $\Pf_{1}$ is given in section \ref{S:stratify.P1}. 
    
For $i \in \NN_{n}$ and $s = \pm 1$, define
    \begin{align}  \label{E:mclX,mclY.defns}
      \mcl{X}_{i,s} &:= \bigl\{ \wrap(\alpha_{1} e_{i} + \alpha_{2} 1_{n} ) \in \Pf_{1i} : 
        s \alpha_{1} \in (0, \pi), \alpha_{2} \in \RR \bigr\} , \\
      \Y_{i} &:= \bigl\{ \wrap(\pi e_{i} + \alpha_{2} 1_{n} ) \in \Pf_{1i} : 
        \alpha_{2} \in \RR \bigr\} \notag .
    \end{align}
Alternatively, we have 
    \begin{align}  \label{E:coord.free.mclX,mclY}
      \mcl{X}_{i,s} &= \bigl\{ y \in \Pf_{1i} : \angle_{s}(y_{i}, y_{j}) \in (0,\pi) \bigr\} , \\
      \Y_{i} &= \bigl\{ y \in \Pf_{1i} : \angle_{s}(y_{i}, y_{j}) = \pi \bigr\} \notag .
    \end{align}

We have
    \begin{equation}
      \Pf_{1} = \left( \bigcup_{i=1}^{n} (\mcl{X}_{i,-1} \cup \mcl{X}_{i,+1}) \right) 
        \cup \left( \bigcup_{i=1}^{n} \Y_{i} \right) \cup \T .
    \end{equation}
The $\mcl{X}_{i,s}$'s, $ \Y_{i}$'s, and $\T$ are (pairwise) disjoint smooth manifolds. This represents $\Pf_{1}$ as the disjoint union of $3n+1$ strata. (See section \ref{SSS:conical.fibers}.). 

\section{Cone bundle over $\mcl{X}_{i,s}$}  \label{S:C[mclX.is]}
Let $x_{1}, \ldots, x_{n} \in S^{1}$, so $(x_{1}, \ldots, x_{n}) \in \D$. 
Define the ``angular mean'' of $x_{1}, \ldots, x_{n} \in S^{1}$ as follows. 
Let $\gamma \in \RR$ be arbitrary. By \eqref{E:modulo}, there exist unique 
$\xi_{1}, \ldots, \xi_{n}$ s.t.\ 
$x_{i} = (\cos \xi_{i}, \sin \xi_{i})$ ($i = 1, \ldots, n$) and 
$\xi_{1} - \gamma, \ldots, \xi_{n} - \gamma \in (-\pi, \pi]$. Thus, 
$\xi_{1}, \ldots, \xi_{n} \in (\gamma - \pi, \gamma + \pi]$. Let $\bar{\xi}$ be the arithmetic mean of $\xi_{1}, \ldots, \xi_{n}$: 
$\bar{\xi} = n^{-1} (\xi_{1} + \cdots + \xi_{n})$. Then the angular mean, $\bar{x}$ 
of $x_{1}, \ldots, x_{n}$ is     
    \begin{equation}  \label{E:bar.x.i.defn}
      \bar{x} = (\cos \bar{\xi}, \sin \bar{\xi}) .
    \end{equation}

Note that $\bar{x}$ is not well-defined: Suppose $n$ is even, 
$x = (x_{1}, \ldots, x_{n}) \in \D$ with half of the $x_{i}$'s at $(1,0)$ and half at 
$(-1,0)$. Then $\xi_{1}, \ldots, \xi_{n} \in \RR$ s.t.\ $x_{i} = (\cos \xi_{i}, \sin \xi_{i})$ 
($i = 1, \ldots, n$) with half of the $\xi_{i}$'s at 0 and half at $\pi$. Then 
$\bar{x} = (0,1)$. Alternatively, we could take half of the $\xi_{i}$'s at 0 and half 
at $-\pi$. Then $\bar{x}= (0,-1)$.

Suppose $x_{1}, \ldots, x_{n}$ lie in an \emph{open} semicircle and let 
$x = (x_{1}, \ldots, x_{n}) \in \D$. In this case say $x$ is ``semicircular''. In fact, $x$ is semicircular if and only if there are $\xi_{1}, \ldots, \xi_{n}$ lying in an open interval of length $\pi$ and $x_{i} = (\cos \xi_{i}, \sin \xi_{i})$ ($i = 1, \ldots, n$).

Recall \eqref{E:directional.T.defn}. If $y \in S^{1}$, let
    \begin{equation}  \label{E:T(y).defn}
      \T(y) = (y, \ldots, y) \in \T .
    \end{equation} 
We have:
  \begin{lemma}  \label{L:mean.of.semicircular.data.sets}
Let $x \in \D$ be semicircular. Then $\bar{x}$ is well-defined and $\T(\bar{x})$ is the unique closest point of $\T$ to $x$. 
  \end{lemma}
``$\bar{x}$ is well-defined'' means that any choice of $\xi_{1}, \ldots, \xi_{n}$ lying in an open interval of length $\pi$ and s.t.\ $x_{i} = (\cos \xi_{i}, \sin \xi_{i})$ ($i = 1, \ldots, n$) leads to the same $\bar{x}$.
  \begin{proof} Suppose $\gamma, \gamma' \in \RR$; $\xi_{1}, \ldots, \xi_{n} \in (\gamma - \pi, \gamma + \pi]$; $\xi'_{1}, \ldots, \xi'_{n} \in (\gamma' - \pi, \gamma' + \pi]$; and
$x_{i} = (\cos \xi_{i}, \sin \xi_{i}) = (\cos \xi'_{i}, \sin \xi'_{i})$ ($i = 1, \ldots, n$). Then for some $k_{1}, \ldots, k_{n} \in \ZZ$ we have $\xi'_{i} = \xi_{i} + 2 k_{i} \pi$. Suppose 
$k_{2} \neq k_{1}$, say. Then 
$2 \pi > |\xi'_{2} - \xi'_{1}| = |\xi'_{2} - \xi'_{1} + 2 (k_{2} - k_{1}) \pi|> 2 \pi - \pi = \pi$. Thus, $x_{1}, \ldots, x_{n}$ do not lie in an open semicircle, even mod $2 \pi$. Contradiction. This proves the claim that $\bar{x}$ is unique. Moreover, $\bar{x}$ is clearly continuous in $x_{1}, \ldots, x_{n}$ that lie in an open semicircle. 

Recall \eqref{E:wrap.defn.2}. 
Let $\xi_{1}, \ldots, \xi_{n}$ lie in an open interval, $I$, of length $\pi$ and suppose 
$x_{i} = (\cos \xi_{i}, \sin \xi_{i})$ ($i = 1, \ldots, n$). 
Let $\omega \in I$ and $w = \wrap( \bar{x} 1_{n})$.  For every $I = 1, \ldots, n$ we have 
$|\omega - \xi_{i}| < \pi$. Therefore, by \eqref{E:geod.dist.on.D.2}, 
$\rho(x, w)^{2} = (\xi_{1} - \omega)^{2} + \cdots + (\xi_{n} - \omega)^{2}$. In particular, this holds with $\omega = \bar{\xi}$ and $w = \bar{x}$. Moreover, a familiar argument shows that, for $\omega \in I$, we have $\rho(x, \bar{x}) \leq \rho(x, w)$ with equality only 
if $w = \bar{x}$. 

But might there be some $\omega \in S^{1} \setminus I$ s.t.\ with 
$w = \wrap(\omega 1_{n})$ s.t.\ $\rho(x, \bar{x}) \geq \rho(x, w)$? Because the distances only depend on angles separating pairs of points, we need only consider we may assume  
$\xi_{1}, \ldots, \xi_{n} \in I = (-\pi/2, \pi/2)$. 
Let $\xi_{min} := \min \{ \xi_{1}, \ldots, \xi_{n} \}$. Define $\xi_{max}$ similarly. If $\xi_{max} = \xi_{min}$ then $\bar{x} = x_{1} = \cdots = x_{n}$ and $\rho(x, \bar{x}) = 0 < \rho(x, w)$. 
Hence, we may assume $\xi_{min} < \xi_{max}$. In fact, WLOG 
$-\pi/2 < \xi_{min} < 0 < \xi_{max} < \pi/2$. 

Notice that if 
$\omega \in [\xi_{min}, \xi_{min}+\pi)$ then $\omega, \xi_{1}, \ldots, \xi_{n}$ lie in an open interval of length $\pi$. Hence, 
$\rho(x, \bar{x} 1_{n}) \leq \rho(x, w)$ with equality only if $w = \bar{x}$. Similarly if 
$\omega \in (\xi_{max} - \pi, \xi_{min}]$. 

That leaves $\omega \in (-\pi, \xi_{max} - \pi) \cup [\pi, \xi_{min}+\pi)$. In fact, by flipping the points over the horizontal axis, we \emph{really} only need consider 
$\omega \in (\xi_{min}+\pi, \pi]$. Now, $\xi_{min} \in (-\pi/2, 0)$, 
so $\omega > \pi/2$. Let $\zeta \in [0, \pi/2)$ be $\omega$ reflected past $\pi/2$ over the vertical axis: 
$\zeta := \pi/2 - (\omega - \pi/2) = \pi - \omega \in (0, \pi/2)$. Thus, 
$\zeta < \pi/2 < \omega$ and 
$\zeta, \xi_{1}, \ldots, \xi_{n}$ lie in the open interval $(-\pi/2, \pi/2)$ of length $\pi$. Let 
$z := \wrap(\zeta 1_{n})$. Then we know that $\rho(x, z) \geq \rho(x, \bar{x})$. 

\emph{Claim:} $\rho(x, z) < \rho(x, w)$. Thus, $ \rho(x, \bar{x}) <  \rho(x, w)$. Suppose 
$\xi_{k} \geq \zeta$ so $\omega - \xi_{k} < \pi$ and 
    \begin{equation*}
      \zeta \leq \xi_{k} < \pi/2 < \omega .
    \end{equation*} 
Therefore,
    \begin{equation*}
      0 \leq \xi_{k} - \zeta < \pi/2 - \zeta = \omega - \pi/2 < \omega - \zeta 
        \text{ if } \zeta \leq \xi_{k} < \pi/2 .
    \end{equation*}

Now suppose $\omega - \pi \leq \xi_{k} < \zeta$. Now $0 < \omega - \xi_{k} \leq \pi$. Trivially, $\omega - \xi_{k} > \zeta - \xi_{k} > 0$. 

Finally, suppose $\xi_{k} \in [\xi_{min}, \omega - \pi)$. Then $\omega - \xi_{k} > \pi$. So instead we consider $\xi_{k} + 2 \pi - \omega < (\omega - \pi) + 2 \pi - \omega = \pi$. 
Now $\xi_{k} \in (-\pi/2, 0]$ so $2 \xi_{k} > -\pi$ so $\xi_{k} + \pi > - \xi_{k}$. Therefore, 
$\xi_{k} + 2 \pi - \omega = (\xi_{k} + \pi) + (\pi - \omega) > \zeta - \xi_{k} \in (0, \pi)$, as desired. 

We have shown that in every case $\angle(w_{k}, x_{k}) > \angle(z_{k}, x_{k})$. Therefore, by \eqref{E:geod.dist.on.D.2}, $\rho(x, z) < \rho(x, w)$, as claimed. 
  \end{proof}

Let $s = \pm 1$ and $i \in \NN_{n}$. (See \eqref{E:N.sub.n}.) But WLOG $s = +1$. 

Recall
    \begin{equation} \label{E:parametrization.of.T}
      \T \text{ is parametrized by } \phi \mapsto \, \wrap(\phi 1_{n}) 
        \quad (\phi \in \RR) .
    \end{equation} 

Let $s = \pm 1$ be fixed.

Suppose $y = (y_{1}, \ldots, y_{n}) \in \Pf_{1i}$ then, for some 
$\alpha_{1}, \alpha_{2} \in \RR$ we have 
$y = \wrap(\alpha_{1} e_{i} + \alpha_{2} 1_{n} )$. If 
$y \in \mcl{X}_{i,s}$. Then, by \eqref{E:mclX,mclY.defns}, $s \alpha_{1} \in (0, \pi)$. Hence,  the components of $y$ all lie in an open semicircle. Let $\beta := n^{-1} \bigl[ \alpha_{1} + (n-1) \alpha_{2} \bigr]$. Recall \eqref{E:T(y).defn}. Then, by lemma \ref{L:mean.of.semicircular.data.sets},   
         \begin{multline}  \label{E:when.y.in.mclX}
           \text{the unique closest point of $\T$ to $y$ is } \T(\cos \beta, \sin \beta).
         \end{multline}
Employing \eqref{E:s.directed.angle}, \eqref{E:directed.angle.is.Lip}, and \eqref{E:arg.defn}, we get,  
    \begin{multline}  \label{E:alphas.from.y}
      \alpha_{1} = s \, \angle_{s}(y_{j},y_{i}) \text{ for any } j \neq i . \;  
        \alpha_{1} \text{ is unique and Lipschitz in } y . \\ 
          \alpha_{2} = \arg_{\alpha_{1}}(y_{j}) \; (j \neq i)
            \text{ is also unique, $\mod 2 \pi$, and Lipschitz in } y . 
    \end{multline} 

Choose
    \begin{equation}  \label{E:r.defn} 
       r \in (0, \pi/4] .
    \end{equation}

Define 
    \begin{equation}  \label{E:Gi.defn} 
      \mcl{G}_{i} := \bigl\{ x \in \D : \text{ if $j \neq i$, $k \neq i$ then } 
        \angle(x_{j}, x_{k}) < 2 r \bigr\} .
    \end{equation}
By \eqref{E:D.has.prod.top}, 
    \begin{equation}  \label{E:Gi.is.open}
      \mcl{G}_{i} \text{ is open.}
    \end{equation} 

If $\xi \in \RR^{n}$, let $\bar{\xi}^{i} := (n-1)^{-1} \sum_{k \neq i} \xi_{k}$. \emph{Claim:} Suppose $x \in \mcl{G}_{i}$ then 
    \begin{multline}  \label{E:xi.parallels.x.on.Gi}
      \text{If $x \in \mcl{G}_{i}$ there exists $\xi \in \RR^{n}$ s.t.\ } 
        x = \wrap(\xi) \text{ and } |\xi_{j} - \xi_{k}| < 2 r \text{ for every } j, k \neq i . \\
          \text{ Moreover, } \xi_{i} \in (\bar{\xi}^{i} - \pi, \bar{\xi}^{i} + \pi] .
    \end{multline}
Let $x \in \mcl{G}_{i}$. Pick $\ell \neq i$ and $\xi_{\ell} \in \RR$ s.t.\ 
$x_{\ell} = (\cos \xi_{\ell}, \sin \xi_{\ell})$. Recall \eqref{E:arg,matey!.2}. Since 
$x \in \mcl{G}_{i}$, no component of $x^{i}$ equals $-x_{\ell}$. Let 
$C := S^{1} \setminus \{-x_{\ell}\}$ so $x^{i} \in C^{n-1}$. 
$\blds{\arg}_{\xi_{\ell} 1_{n-1}}(x)$ consists of $n-1$ points 
$\xi_{j} \in (\xi_{\ell} - \pi, \xi_{\ell} + \pi)$ ($j \neq i$). Also 
$|\xi_{j} - \xi_{k}| = \angle(x_{j}, x_{k})$ for every $j, k \neq i$. (By \eqref{E:modulo}, we may pick $\xi_{i} \in (\xi_{\ell} - \pi, \xi_{\ell} + \pi]$ s.t.\ 
$x_{i} = (\cos \xi_{i}, \sin \xi_{i})$.) This proves the claim. 

Write $\D_{n} = \D = (S^{1})^{n}$ so $\D_{n-1} = (S^{1})^{n-1}$. Define 
$\T_{n-1} := \bigl\{ (w, \ldots, w) \in \D_{n-1} : w \in S^{1} \bigr\}$. 
Define a metric, $\rho_{n-1}$, on $\D_{n-1}$ as in \eqref{E:geod.dist.on.D.2}. 
Let $x \in \mcl{G}_{i}$. Then $x^{i}$ is semicircular in $\D_{n-1}$. Therefore, by lemma \ref{L:mean.of.semicircular.data.sets} applied to $\D_{n-1}$, the angular mean, 
$\bar{x}^{i}$ exists uniquely and $\T_{n-1}(\bar{x}^{i})$ is the unique closest point 
of $\T_{n-1}$ to $x^{i}$. Let $y \in \Pf_{1i}$. 
    \begin{equation*}
      \rho(x,y)^{2} = \angle(x_{i}, y_{i})^{2} + \rho_{n-1}(x^{i}, y^{i})^{2} .
    \end{equation*}
Recall \eqref{E:set.distances}. For simplicity, write $dist = dist_{\rho}$. Now, $y^{i} \in \T_{n-1}$ so  
    \begin{equation*}
      \rho(x,y)^{2} \geq dist(x^{i}, \T_{n-1})^{2} .
    \end{equation*}
Therefore, 
    \begin{multline}  \label{E:x.in.Gi.has.unique.closest.pt.in.Pf.1i}
      \text{Suppose $x \in \mcl{G}_{i}$. Then } y = (y_{1}, \ldots, y_{n}) 
        \text{ with $y_{i} = x_{i}$ and } y^{i} = \T_{n-1}(\bar{x}^{i}) \\
          \text{ is the unique closest point of $\Pf_{1i}$ to } x .
    \end{multline} 

We will often assume or deduce the following.
       \begin{equation}  \label{E:x.near.x.i.bar}
         x \in \mcl{G}_{i}, \; \angle_{s}(x_{i}, \bar{x}^{i}) \in (0, \pi), 
           \text{ and } \angle(x_{j}, \bar{x}^{i}) < r \text{ for } j \neq i.
       \end{equation}
(If $x$ satisfies the preceding then $\angle_{(-s)}(x_{j}, \bar{x}^{i})) > 2 \pi - r > \pi$.) Recall \eqref{E:mclX,mclY.defns}. From what we saw above, we have, 
    \begin{equation}  \label{E:x.near.x.i.bar.is.close.to.Xis}
      \text{If $x$ satisfies \eqref{E:x.near.x.i.bar} then }
        \text{the closest point of $\Pf_{1i}$ to $x$ lies in } \mcl{X}_{i,s} .
    \end{equation}

For $j \in \NN_{n}$, 
	\begin{multline}   \label{E:e.sub.i.defn}
		\text{Let } e_{j}^{1 \times n} \text{ be the coordinate vector } 
		  (0, \ldots, 0, 1, 0, \ldots, 0) \in \RR^{n} \\
		  \text{ with the ``1'' in the $j^{th}$ position.}
	\end{multline}
Recall the definition, \eqref{E:1n.col.vec.defn}, of $1_{n}$. Define: 
     \begin{equation}  \label{E:OneNbSpace.defn.3} 
      \OneNbSpace_{i} \subset \RR^{n} 
        \text{ is the space spanned by $e_{i}$ and $1_{n}$.}
    \end{equation}
Thus, $\beta \in \OneNbSpace_{i}$. Let 
    \begin{equation}  \label{E:fi.defn.2}
      f_{i}^{1 \times n} := 1_{n} - e_{i} , \quad i \in \NN_{n} .
    \end{equation}
Thus, $f_{ii} = 0$ and $f_{ij} = 1$ ($j \neq i$). Thus, $f_{i} \cdot 1^{n} = n-1$, 
$f_{i} \cdot e_{i} = 0$, and $\OneNbSpace_{i}$ is also spanned by $e_{i}$ and $f_{i}$. 

Let $x = (x_{1}, \ldots, x_{n}) \in \mcl{G}_{i} \subset \D$. Then, by \eqref{E:wrap.defn.2}, for some $\xi = (\xi_{1} , \ldots, \xi_{n}) \in \RR^{n}$ we have  
	\begin{equation} \label{E:wrap.defn.3}
		x = \wrap(\xi) := ( \cos \xi_{1}, \sin \xi_{1}, 
		  \cos \xi_{2}, \sin \xi_{2}, 
			\ldots, \cos \xi_{n}, \sin \xi_{n} \bigr) .
	\end{equation}

Let 
    \begin{equation}  \label{E:delta.i.defn.2}
      \bar{\xi}^{i} := (n-1)^{-1} (\xi \cdot f_{i}) \in \RR \text{ and }
        \deli := \deli(\xi) := \xi - \bar{\xi}^{i} 1_{n} .
    \end{equation}
Thus, $\bar{\xi}^{i}$ is the arithmetic mean of the coordinates of $\xi$, excluding 
the $i^{th}$, and 
    \begin{equation}  \label{E:delta.i.from.xi}
      \deli_{k} := \deli(\xi) := \xi_{k} - \bar{\xi}^{i}, \quad k \in \NN_{n} . 
    \end{equation}
Alternatively, 
    \begin{equation}  \label{E:coord.free.delta}
      \deli_{k} = s \angle_{s}(\bar{x}^{i}, x_{k}), \quad k = 1, \ldots, n . 
    \end{equation} 
Since $x = (x_{1}, \ldots, x_{n}) \in \mcl{G}_{i}$, by \eqref{E:Gi.defn}, \eqref{E:modulo}, and lemma \ref{L:mean.of.semicircular.data.sets}, we may assume 
    \begin{multline*}
      \xi_{i} \in (\bar{\xi}^{i} - \pi , \bar{\xi}^{i} + \pi] , \text{ and } |\xi_{j} - \xi_{k}| < 2 r
        \text{ for } j \neq i, k \neq i . \\
          \text{ Moreover, the closest point of $\Pf_{1i}$ to $x$ is }
            \wrap(\xi_{i} e_{i} + \bar{\xi}^{i} f_{i}) .           
    \end{multline*}

Note that
    \begin{equation}  \label{E:non-i.deltas.sum.to.0}
      f_{i} \cdot \deli = f_{i} \cdot (\xi - \bar{\xi}^{i} 1_{n})
        = (n-1) \bar{\xi}^{i} - (n-1) \bar{\xi}^{i} = 0. \; \text{ So }
          \sum_{k \neq i} \deli_{k} = 0 .
    \end{equation}
    
We can re-express \eqref{E:x.near.x.i.bar} as follows.
        \begin{multline}  \label{E:xk.near.xi.i.bar} 
          x = \wrap(\xi), \text{ } 
            \xi_{i} \in ( 0, \bar{\xi}^{i} + \pi ) \text{ and }
              \xi_{k} \in (\bar{\xi}^{i} - r, \bar{\xi}^{i} + r) \text{ if } k \neq i . \\
                \text{ In terms of $\deli$ this becomes: }
                  0 < \deli_{i} < \pi \text{ and } |\deli_{k}| < r  \text{ if } k \neq i .
        \end{multline}
Write $\xi = \xi(x)$. Note that if $x$ satisfies \eqref{E:xk.near.xi.i.bar}, then, by lemma \ref{L:mean.of.semicircular.data.sets}, $\bar{x}^{i}$ is unique. 
 
Let $\phi \in \RR$ and let  
$C = \bigl\{ (\cos \theta, \sin \theta) : |\theta - \phi| < \pi \bigr\}$. 
Then, by \eqref{E:arg.is.isom}, $\blds{\arg}_{\phi 1^{n}}$ is Lipschitz on $C^{n}$. Hence,   
    \begin{multline}  \label{E:xi.is.Lip}
      \text{Given $\phi \in \RR$, there is a well-defined function } \\
        \xi(\cdot) = \xi_{\phi}(\cdot) 
          : C^{\,n} \to (\phi - \pi, \phi + \pi)^{n} 
            \text{ satisfying } \wrap(\xi(x)) = x , \text{ for } x \in C^{\,n} . \\
              \xi(\cdot) \text{ is Lipschitz.}
    \end{multline}  

If $\xi \in \RR^{n}$, let $\xi^{i} \in \RR^{n-1}$ be $\xi$ with the $i^{th}$ coordinate dropped. Interpret $\wrap$ \, in the obvious way as a map from 
$\RR^{n-1} \to (S^{1})^{n-1}$. Then the same argument as above shows   
    \begin{multline}  \label{E:xi.i.is.Lip}
      \text{Given $\phi \in \RR$, there is a well-defined function } \\
        \xi^{i}(\cdot) = \xi^{i}_{\phi}(\cdot) 
          : C^{\,n-1} \to (\phi - \pi, \phi + \pi)^{n-1} 
            \text{ satisfying } \wrap(\xi^{i}(x)) = x^{i} , \text{ for } x \in C^{\,n-1} . \\
              \xi^{i}(\cdot) \text{ is Lipschitz.}
    \end{multline}

For $x \in \D$ let 
$\mcl{E}^{i}(x) := \bigl\{ z \in \D : \angle(z_{j}, x_{j}) < \pi/4 , \; j \neq i \bigr\}$ so 
$\mcl{E}^{i}(x)$ is an open neighborhood of $x$. \emph{Claim:}    
    \begin{multline}  \label{E:locally.Lip.xbari.on.Gi}   
      \text{Given $x \in \mcl{G}_{i}$ there is a well-defined Lipschitz version of the map } \\
        z \to \bar{z}^{i} \in S^{1} \text{ on } \mcl{E}(x) . 
    \end{multline} 
Let $x \in \mcl{G}_{i}$ and let $z \in \mcl{E}^{i}(x)$. Pick 
$\tilde{\xi}, \tilde{\zeta} \in \RR^{n}$ s.t.\ $|\tilde{\zeta}_{j} - \tilde{\xi}_{j}| < \pi/4$ for every 
$j \neq i$. By \eqref{E:xi.parallels.x.on.Gi}, there exists $\ell \in \ZZ^{n}$ s.t.\ if 
$\xi := \tilde{\xi} + 2 \pi \ell$ then 
$|\xi_{j} - \xi_{k}| < 2 r$ for every $j, k \neq i$. Let $\zeta := \tilde{\zeta} + 2 \pi \ell$. 
Then $\wrap(\xi) = x$, $\wrap(\zeta) = z$, and $|\zeta_{j} - \xi_{j}| < \pi/4$ for every 
$j \neq i$. Let $\xi_{min} := \min \{ \xi_{j} : j \neq i \}$ and 
$\xi_{max} := \max \{ \xi_{j} : j \neq i \}$. Since $x \in \mcl{G}_{i}$, by \eqref{E:r.defn}, 
$\xi_{max} - \xi_{min} < \pi/2$. Hence, if $j \neq i$ then 
$\xi_{min} - \pi/4 < \zeta_{i} < \xi_{max} + \pi/4 < \xi_{min} + \pi/2 + \pi/4$. I.e., 
$\zeta_{j} \in (\xi_{min} - \pi/4 , \xi_{min} + 3 \pi/4)$, an open interval of length $\pi$. 

By \eqref{E:coord.free.delta} and \eqref{E:xi.i.is.Lip},    
    \begin{equation}  \label{E:Lip.delta.k.on.Gi}   
      \text{The maps } \deli_{k}(\cdot) \; (k \neq i)
        \text{ are Lipschitz on }  \mcl{G}_{i} .
    \end{equation} 
 
We have
   \begin{lemma}  \label{L:narrow.n-1.means.rho.closest}
Suppose $\xi, \xi' \in \RR^{n}$, $\wrap(\xi') = \, \wrap(\xi) = x \in \D$. Let $r$ satisfy \eqref{E:r.defn}. Assume \eqref{E:xk.near.xi.i.bar} 
(for $x$, not necessarily for $x'$).  
Then $dist(\xi', \OneNbSpace_{i}) \leq dist(\xi, \OneNbSpace_{i})$ if and only if 
        \begin{equation}  \label{E:add.mults.of.2pi.to.xi}
          \xi' = \xi + 2 \pi M f_{i} + 2 \pi N e_{i} , 
        \end{equation}
where $M,N \in \ZZ$. In that case, we have
        \begin{equation*}
          dist(\xi', \OneNbSpace_{i}) = dist(\xi, \OneNbSpace_{i})
            = dist(x, \Pf_{1i}) . 
        \end{equation*}
and $y(x) := \wrap(\xi_{i} e_{i} + \bar{\xi}^{i} f_{i})$ is the unique closest point 
of $\Pf_{1i}$ to $x$.
  \end{lemma}

It follows from the preceding and \eqref{E:xi.is.Lip} that $y(x)$ is continuous in $x$ for which \eqref{E:x.near.x.i.bar} (equivalently \eqref{E:xk.near.xi.i.bar}) holds. 
  \begin{proof}
WLOG $i = 1$. Since $\wrap(\xi') = \, \wrap(\xi) = x \in \D$, by \eqref{E:wrap.defn.2}, 
the coefficient of $e_{1}$ in \eqref{E:add.mults.of.2pi.to.xi} has to have the form $2 \pi N$, with $N \in \ZZ$.

Now we consider the coefficient of $f_{1}$. Write 
    \begin{equation*}
      \xi = (\xi_{1}, \gamma), \text{ where } \gamma \in \RR^{n-1} .
    \end{equation*} 
Let $\bar{\gamma} := (n-1)^{-1} (\gamma_{1} + \cdots + \gamma_{n-1}) 
= (n-1)^{-1} \gamma \cdot 1_{n-1} = \bar{\xi}^{1}$. Recall \eqref{E:defn.of.Pi(xi)}. 
(``$\xi$'' means something different there than it does here.) By \eqref{E:OneNbSpace.defn.3}, 
$e_{1} \in \OneNbSpace_{1}$. Therefore, by \eqref{E:Pf.1j.rho.dist}, the Euclidean distance, 
$dist(\xi, \OneNbSpace_{1})$, from $\xi$ to $\OneNbSpace_{1}$ is
    \begin{equation}  \label{E:dist1.formla}
      dst_{1} := \bigl| \xi - \xi \, \Pi(\OneNbSpace_{1}) \bigr| 
        = \sqrt{\sum_{k=2}^{n} (\xi_{k}^{2} - \bar{\xi}^{1})^{2}}
          = \sqrt{\sum_{\ell =1}^{n-1} (\gamma_{\ell}^{2} - \bar{\gamma})^{2}} .
    \end{equation}
Since $\wrap(\xi') = \, \wrap(\xi) = x$ by assumption, there exists 
$\alpha \in \ZZ^{n}$ s.t.\ $\xi' := \xi + 2 \pi \alpha$. Write $\alpha = (\alpha_{1}, m)$, 
so $m \in \ZZ^{n-1}$. So
    \begin{equation}  \label{E:xi'.alpha.m}
      \xi' := \xi + 2 \pi \alpha = (\xi_{1} + 2 \pi \alpha_{1}, \gamma + 2 \pi m) .
    \end{equation}
Let
    \begin{equation}  \label{E:m.bar.and.zeta}
      \overline{m} := \frac{1}{n-1} m \cdot 1_{n-1} \in \RR 
        \text{ and } \zeta := \gamma - \bar{\gamma} 1_{n-1} 
          = \gamma - \bar{\xi}^{1} 1_{n-1} \in \RR^{n-1} .
    \end{equation}
Thus, by \eqref{E:dist1.formla}, 
    \begin{equation}  \label{E:dist1=|zeta|}
      dst_{1} = |\zeta| .
    \end{equation} 
By \eqref{E:non-i.deltas.sum.to.0}, 
    \begin{equation}  \label{E:resids.sum.to.0}
      \zeta \cdot 1_{n-1} = 0 = (m - \overline{m} 1_{n-1}) \cdot 1_{n-1} .
    \end{equation}
It follows that
    \begin{equation}  \label{E:no.product.term}
      |m - \overline{m} 1_{n-1}|^{2} = |m|^{2} - (n-1) \overline{m}^{2} .
    \end{equation}

By \eqref{E:modulo}, we may write 
    \begin{equation}  \label{E:m,p,eta}
      \overline{m} = p + \eta , 
        \text{ where } p \in \ZZ \text{ and } \eta \in (-1/2, 1/2] . 
    \end{equation} 
Let 
    \begin{equation}  \label{E:L=m-p.1.n-1}
      L := m - p 1_{n-1} \in \ZZ^{n-1} .
    \end{equation}
Then, by \eqref{E:resids.sum.to.0}, 
    \begin{multline}  \label{E:L.dot.1=(n-1)eta}
      L \cdot 1_{n-1} = (m - p 1_{n-1} - \eta 1_{n-1}) \cdot 1_{n-1} + (n-1) \eta \\
        = (m - \overline{m} 1_{n-1}) \cdot 1_{n-1} + (n-1) \eta = (n-1) \eta .
    \end{multline}
 
Recall that, e.g., $| L |_{1} := |L_{1}| + \cdots +  |L_{n-1}|$. \emph{Claim:}
    \begin{equation}  \label{E:L1.norm.at.least.(n-1)eta}
      |L|_{1} \geq (n-1)|\eta| .
    \end{equation}
If $\eta = 0$ and $L \neq 0$, then the inequality is obviously strict. 
So assume $\eta \neq 0$. We have
    \begin{equation*}
      |L|_{1} = \sum_{m_{k} \geq p} (m_{k} - p) - \sum_{m_{k} < p} (m_{k} - p) .
    \end{equation*}
Suppose $\eta > 0$. Then, by the preceding and \eqref{E:L.dot.1=(n-1)eta}, 
    \begin{align*}
      |L|_{1} &= \left[ \sum_{m_{k} \geq p} (m_{k} - p) 
        + \sum_{m_{k} < p} (m_{k} - p) \right]
          - 2 \sum_{m_{k} < p} (m_{k} - p) \\
      &= [L \cdot 1_{n-1}] - 2 \sum_{m_{k} < p} (m_{k} - p) \\
      &= (n-1) \eta - 2 \sum_{m_{k} < p} (m_{k} - p) \\
      &\geq (n-1) |\eta| .
    \end{align*}
Similarly, suppose $\eta < 0$. Then, by \eqref{E:L.dot.1=(n-1)eta}, 
    \begin{align*}
      |L|_{1} &= - \left[ \sum_{m_{k} \geq p} (m_{k} - p) 
        + \sum_{m_{k} < p} (m_{k} - p) \right]
          + 2 \sum_{m_{k} \geq p} (m_{k} - p) \\
      &= -L \cdot 1_{n-1} + 2 \sum_{m_{k} \geq p} (m_{k} - p) \\
      &= -(n-1) \eta + 2 \sum_{m_{k} \geq p} (m_{k} - p) \\
      &\geq (n-1) |\eta| .
    \end{align*}
This proves the claim \eqref{E:L1.norm.at.least.(n-1)eta}.

Let $dst'_{1} := dist(\xi', \OneNbSpace_{1})$. Then, 
by \eqref{E:xi'.alpha.m}, \eqref{E:m.bar.and.zeta}, \eqref{E:m,p,eta}, 
\eqref{E:L=m-p.1.n-1}, \eqref{E:dist1=|zeta|}, \eqref{E:resids.sum.to.0}, and \eqref{E:L.dot.1=(n-1)eta}, we have
    \begin{align}  \label{E:dst.tilde.dst}
      (dst'_{1})^{2} &= | (\gamma + 2 \pi m) - (\bar{\gamma} 1_{n-1} 
          + 2 \pi \overline{m} 1_{n-1}) |^{2} \notag \\
        &= \bigl| \zeta + 2 \pi (m - \overline{m} 1_{n-1}) \bigr|^{2} \notag \\
        &= \bigl| \zeta + 2 \pi (m - p 1_{n-1} - \eta 1_{n-1}) \bigr|^{2} \notag \\
        &= \bigl| \zeta + 2 \pi (L - \eta 1_{n-1}) \bigr|^{2} \notag \\
        &= |\zeta|^{2} + 4 \pi \zeta \cdot (L - \eta 1_{n-1}) 
          + 4 \pi^{2} |L - \eta 1_{n-1}|^{2} \\
        &= dst_{1}^{2} + 4 \pi \zeta \cdot L + 4 \pi^{2} |L - \eta 1_{n-1}|^{2} \notag \\
        &= dst_{1}^{2} + 4 \pi \zeta \cdot L + \Bigl( 4 \pi^{2} |L|^{2} 
          - 8 \pi^{2} \eta (L \cdot 1_{n-1}) + 4 \pi^{2} (n-1) \eta^{2} \Bigr) \notag \\
        &= dst_{1}^{2} + 4 \pi \zeta \cdot L + 4 \pi^{2} |L|^{2} 
          - 8 \pi^{2} (n-1) \eta^{2} + 4 \pi^{2} (n-1) \eta^{2} \notag \\
        &= dst_{1}^{2} + 4 \pi \zeta \cdot L + 4 \pi^{2} |L|^{2} 
          - 4 \pi^{2} (n-1) \eta^{2} \notag .
    \end{align}
By the H\"{o}lder inequality (Rudin \cite[Theorem 3.8, p.\ 65]{wR66.realcmplx}),
we have $\zeta \cdot L \geq - |\zeta |_{\infty} |L|_{1}$. By assumption, 
\eqref{E:xk.near.xi.i.bar} holds for $\xi$. Thus, 
    \begin{equation}  \label{E:|zeta|.infty.<2r}
      |\zeta |_{\infty} < 2 r ,
    \end{equation} 
a strict inequality. Thus, we have
    \begin{multline}  \label{E:hold.her}
      (dst'_{1})^{2} \geq dst_{1}^{2} - 4 \pi |\zeta |_{\infty} |L|_{1} 
        + 4 \pi^{2} |L|^{2} - 4 \pi^{2} (n-1) \eta^{2} \\
          \geq dst_{1}^{2} - 8 \pi r |L|_{1} + 4 \pi^{2} |L|^{2} - 4 \pi^{2} (n-1) \eta^{2} . \\
      \text{ If $L \neq 0$ the inequality is strict.}
    \end{multline}

Since $p$ and the coordinates of $m$ are integers we have, 
by \eqref{E:L1.norm.at.least.(n-1)eta},  
    \begin{equation}  \label{E:L2.L1.eta}
      |L|^{2} = \sum_{k} (m_{k} - p)^{2} \geq \sum_{k} |m_{k} - p| = |L|_{1} 
        \geq (n-1) |\eta| . 
    \end{equation}
Recall that $|\eta| \leq 1/2$ and, by \eqref{E:r.defn}, $r \leq \pi/4$. Substituting \eqref{E:L2.L1.eta} into \eqref{E:hold.her} and applying \eqref{E:|zeta|.infty.<2r}, we get
    \begin{align}  \label{E:|L|1.dist1}
      (dst'_{1})^{2} &\geq dst_{1}^{2} + 4 \pi ( - 2 r + \pi ) |L|_{1}
        - 4 \pi^{2} (n-1) \eta^{2} \notag \\
        &\geq dst_{1}^{2} + 4 \pi ( \pi/2 ) |L|_{1}
        - 4 \pi^{2} (n-1) \eta^{2} \notag \\
      &\geq dst_{1}^{2} + 2 \pi^{2} |L|_{1}
        - 2 \pi^{2} (n-1) |\eta | \\
      &\geq dst_{1}^{2} + 2 \pi^{2} |L|_{1}
        - 2 \pi^{2} |L|_{1} \notag \\
      &= dst_{1}^{2} \notag .
    \end{align} 

If $m = M 1_{n-1}$ (which, by \eqref{E:xi'.alpha.m}, translates to $M f_{i}$ in \eqref{E:add.mults.of.2pi.to.xi}) for some $M \in \ZZ$ then obviously $dst'_{1} = dst_{1}$.  (The first coordinate 
in $\xi'$ has no bearing on $dst'_{1}$.) By \eqref{E:L=m-p.1.n-1} and \eqref{E:m,p,eta}, $m$ is a multiple of $1_{n-1}$ if and only if $L = 0$. And $L = 0$ if and only if 
$|L|_{1} = 0$. Suppose $m$ is not a multiple of $1_{n-1}$ so $|L|_{1} > 0$. Then, by \eqref{E:hold.her}, if $|L|_{1} > 0$ the first inequality in \eqref{E:|L|1.dist1} is strict. Thus, $m$ not a multiple of $1_{n-1}$ means $dst'_{1} > dst_{1}$, contradicting $dist(\xi', \OneNbSpace_{i}) \leq dist(\xi, \OneNbSpace_{i})$. 
\eqref{E:add.mults.of.2pi.to.xi} follows.

Suppose $y' \in \Pf_{1i}$ is a closest point of $\Pf_{1i}$ to $x$ and write $y' = \, \wrap(\beta')$ where $\beta' \in \OneNbSpace_{1}$. There exists $\xi' \in \ \RR^{n}$ s.t.\ $\wrap(\xi') = x$ and $|\xi' - \beta'| = \rho(x,y')$. By \eqref{E:add.mults.of.2pi.to.xi}, which we just proved, 
$\xi' = \xi + 2 \pi M f_{i} + 2 \pi N e_{i}$, where $M,N \in \ZZ$. Therefore,   
$\beta' = \xi'_{i} e_{i} + (\bar{\xi}')^{i} f_{i} = (\xi_{i} + 2 \pi N) e_{i} 
+ (\bar{\xi}^{i} + 2 \pi M ) f_{i}$. Hence, each coordinate of $\beta'$ differs from the corresponding coordinate of $\beta$ by an integral multiple of $2 \pi$. Therefore $\wrap(\beta') = \, \wrap(\beta)$ and the lemma is proved. 
  \end{proof}  

We also have
   \begin{lemma}  \label{L:dist.infty.x.Pf.1j<r.rho.dist.computible.from.xi}
Let $i \in \NN_{n}$. Suppose $\xi \in \RR^{n}$, $x = \wrap(\xi) \in \D$, and $\xi$ satisfies \eqref{E:xk.near.xi.i.bar}. Thus, by lemma \ref{L:narrow.n-1.means.rho.closest}, 
$dist(\xi, \OneNbSpace_{i}) = dist(x, \Pf_{1i})$. Let $j \in \NN_{n}$. Then, whether $j = i$ or not, we have 
        \begin{equation}  \label{E:Pf.1j.rho.dist}
          dist(x, \Pf_{1j})
            = \sqrt{\sum_{k \neq j}^{n} (\xi_{k} - \bar{\xi}^{j})^{2}} 
              = dist(\xi, \OneNbSpace_{j}) .
       \end{equation}
I.e., the same $\xi$ used to calculate the $\rho$ distance to $\Pf_{1i}$ can  also be used to calculate the $\rho$ distance to $\Pf_{1j}$. 
  \end{lemma} 
  \begin{proof} 
The $j = i$ case is trivial. Assume $j \neq i$. Choose $\xi' \in \RR^{n}$ s.t.\ 
$x = \wrap(\xi') \in \D$, and $\xi'$ satisfies \eqref{E:xk.near.xi.i.bar} with $\xi'$ in place 
of $\xi$ and $j$ in place of $j$. 
\emph{A fortiori} \eqref{E:xk.near.xi.i.bar} similarly also holds for $j$ and $\xi'$. Therefore, by lemma \ref{L:narrow.n-1.means.rho.closest}, 
    \begin{equation}  \label{E:dists.come.from.xi'}
      dist(x, \Pf_{1j}) = \sqrt{ \sum_{k \neq j} (\xi'_{k} - \overline{\xi'}^{j})^{2} } .
    \end{equation}

By \eqref{E:xk.near.xi.i.bar},
    \begin{multline}  \label{E:off.i.xi.diff}
      \text{if } k \neq i, \; \ell \neq i \text{ then }
        |\xi_{k} - \xi_{\ell}| \leq |\xi_{k} - \bar{\xi}^{i}| + |\bar{\xi}^{i} - \xi_{\ell}| < 2r . \\
          \text{ Similarly, } \text{if } k \neq j, \; \ell \neq j \text{ then } 
            |\xi'_{k} - \xi'_{\ell}| < 2r .
    \end{multline}
Since $\wrap(\xi') = x = \, \wrap(\xi)$, there exist $m_{1}, \ldots, m_{n} \in \ZZ$ s.t.\ 
$\xi'_{k} = \xi_{k} + 2 m_{k} \pi$ ($k \in \NN_{n}$). Let $N := \NN_{n} \setminus \{i,j\}$. 
By \eqref{E:n.nvar.k.sizes}, $N \neq \varnothing$. Suppose for the moment $n > 3$ so the cardinality of $N$ is at least 2. Let $k, \ell \in N$ be distinct. 
Then, by \eqref{E:off.i.xi.diff}, $|\xi_{k} - \xi_{\ell}| < 2 r$ and $|\xi'_{k} - \xi'_{\ell}| < 2 r$. 
Suppose $m_{k} \neq m_{\ell}$. Then, by \eqref{E:r.defn}, 
    \begin{multline*} 
      \pi > 2 r > |\xi'_{k} - \xi'_{\ell}| 
        = \bigl| (\xi_{k} + 2 \pi m_{k}) - (\xi_{\ell} + 2 \pi m_{\ell}) \bigr| \\
          \geq 2 \pi |m_{k} - m_{\ell}| - |\xi_{k} - \xi_{\ell}| > 2 \pi - 2 r
             > \pi .
    \end{multline*}
Contradiction. Thus, $m_{k} = m_{\ell}$ for $k, \ell \in N$. Let $m$ be the common value of $m_{k}$ for $k \in N$. (It is possible that $m_{i} \neq m$ and/or 
$m_{j} \neq m$.) If $n=3$ and $k$ is the lone element of $N$, 
let $m := m_{k}$. 

Whether $n=3$ or not, now replace $\xi'_{\ell}$ by $\xi'_{\ell} - 2 m \pi$ (for all $\ell \in \NN_{n}$). 
So $\overline{\xi'}^{j}$ is replaced by $\overline{\xi'}^{j} - 2 m \pi$ and we have 
$\xi'_{k} = \xi_{k}$ ($k \in N$). Shifting by $2 m \pi$ does not change the differences among 
$\overline{\xi'}^{j}, \xi'_{1}, \ldots, \xi'_{n}$. Therefore, \eqref{E:xk.near.xi.i.bar} continues to hold for $\xi'$ and, since by \eqref{E:OneNbSpace.defn.3} we have 
$m 1_{n} \in \OneNbSpace_{i} \cap \OneNbSpace_{j}$, \eqref{E:dists.come.from.xi'}  also still holds. 

We have shown $\xi'_{k} = \xi_{k}$ ($k \in N$). But $i \notin N$ so \emph{prima facie} we are only entitled to say $\xi'_{i} - \xi_{i} = 2 m_{i} \pi$ for some $m_{i} \in \ZZ$. Let $k \in N$. By \eqref{E:xk.near.xi.i.bar}, $|\xi_{k} - \bar{\xi}^{i}| < r$. Since neither $i$ nor $k$ equals $j$ and $\xi'_{k} = \xi_{k}$, \eqref{E:off.i.xi.diff} tells us that $2r > |\xi'_{k} - \xi'_{i}| = | \xi_{k} - \xi'_{i}|$. Hence, by \eqref{E:xk.near.xi.i.bar} and  \eqref{E:r.defn} again,  
        \begin{equation*}
          2 |m_{i}| \pi = |\xi_{i} - \xi'_{i}| 
            \leq |\xi_{i} - \bar{\xi}^{i}| + |\bar{\xi}^{i} - \xi_{k}| + |\xi_{k} - \xi'_{i}|
              < \pi + r + 2r < 2 \pi .
        \end{equation*}
Therefore, $m_{i} = 0$, so $\xi'_{i} = \xi_{i}$. Thus, $\xi'_{k} = \xi_{k}$ for \emph{all} 
$k \neq j$. This proves \eqref{E:Pf.1j.rho.dist}. 

Let $k \in N$, so $\xi'_{k} = \xi_{k}$. There exists $m_{j} \in \ZZ$ s.t.\ 
$\xi_{j} - \xi'_{j} = 2 m_{j} \pi$. Since $j \neq i$, by \eqref{E:off.i.xi.diff}, we have 
$|\xi'_{k} - \xi_{j}| = |\xi_{k} - \xi_{j}| < 2r$. Therefore, like before, 
        \begin{multline*}
           2 |m_{j}| \pi = |\xi_{j} - \xi'_{j}| 
            \leq |\xi'_{j} -  \overline{\xi'}^{j}| + | \overline{\xi'}^{j} - \xi'_{k}| + |\xi'_{k} - \xi_{j}| \\
              = |\xi'_{j} -  \overline{\xi'}^{j}| + | \overline{\xi'}^{j} - \xi'_{k}| + |\xi_{k} - \xi_{j}| 
                < \pi + r + 2r < 2 \pi .
        \end{multline*}
 Thus, $\xi'_{j} = \xi_{j}$ so after the shift by $2 m \pi$ we have $\xi' = \xi$. Hence, we may assume $\xi' = \xi$ and the lemma is proved. 
  \end{proof}
  
Recall \eqref{E:delta.i.defn.2} and \eqref{E:sign.function}. We have
  \begin{lemma}  \label{L:tween.-(n-2)/n.and.1}
Let $x \in \D$ and let $\xi = \xi(x) \in \RR^{n}$ satisfy \eqref{E:xk.near.xi.i.bar}. 
Let $j \neq i$. By lemmas \ref{L:narrow.n-1.means.rho.closest} and \ref{L:dist.infty.x.Pf.1j<r.rho.dist.computible.from.xi}, 
    \begin{align}  \label{E:dists.come.from.xi}
      dist(x, \Pf_{1i}) &= \sqrt{ \sum_{k \neq i} (\xi_{k} - \bar{\xi}^{i})^{2} } 
        = \sqrt{ \sum_{k \neq i} (\deli_{k})^{2} }  \text{ and } \\
      dist(x, \Pf_{1j}) &= \sqrt{ \sum_{k \neq j} (\xi_{k} - \bar{\xi}^{j})^{2} }
         = \sqrt{ \sum_{k \neq j} (\delta^{j}_{k})^{2} } \notag .
    \end{align}
Then the $dist(x, \Pf_{1i}) \leq dist(x, \Pf_{1j})$ if and only if 
    \begin{equation}  \label{E:delta.(n-2)/n.ineq}
     -\frac{n-2}{n} s \deli_{i} \leq s \deli_{j} \leq s \deli_{i} . 
    \end{equation}
Here, $s = sign(\deli_{i})$. $dist(x, \Pf_{1i}) < dist(x, \Pf_{1j})$ if and only if the inequalities in the preceding are strict. 
  \end{lemma} 
By \eqref{E:xk.near.xi.i.bar}, $-r < \deli_{j} < r$. Therefore, if $|\deli_{i}| \geq nr/(n-2)$ then \eqref{E:delta.(n-2)/n.ineq} automatically holds. Thus, \eqref{E:delta.(n-2)/n.ineq} is only potentially interesting when $|\deli_{i}|$ is small. 

  \begin{proof}[Proof of lemma \ref{L:tween.-(n-2)/n.and.1}] 
Let $\eta^{i} \in \RR^{n-1}$ be the vector $\deli$ with the $i^{th}$ coordinate deleted. By \eqref{E:dists.come.from.xi}, the geodesic distance from $x$ to $\Pf_{1i}$ is just the Euclidean length $|\eta^{i}|$. WLOG $i = 1$, $j=2$. To simplify the notation slightly, 
write $\epsilon = \delta^{1}$. Let $\gamma \in \RR^{n-2}$ be the vector obtained by dropping the first two coordinates of $\epsilon$. 
Thus, 
    \begin{equation*}
      \delta^{1} = \epsilon = (\epsilon_{1}, \eta^{1}) 
        = (\epsilon_{1}, \epsilon_{2} , \gamma) . 
    \end{equation*}
By \eqref{E:non-i.deltas.sum.to.0}, 
    \begin{equation}  \label{E:delta.i.2=gamma.sum}
      \epsilon_{2} = - \sum_{k > 2} \epsilon_{k} = - \gamma \cdot 1_{n-2} .
    \end{equation}
Thus, by \eqref{E:dists.come.from.xi}, the squared geodesic distance from $x$ to $\Pf_{11}$ is 
    \begin{equation}  \label{E:dist1.squared}
      dst_{1}^{2} := dist(x, \Pf_{11})^{2}
        = |\eta^{1}|^{2} = \epsilon_{2}^{2} + |\gamma |^{2} 
          = (\gamma \cdot 1_{n-2})^{2} + |\gamma |^{2}  . 
    \end{equation}

In order to compute the corresponding squared distance from $x$ to $\Pf_{12}$ we need $\delta^{2}$. $\delta^{2}$ can be computed from $\xi$. But $\delta^{2}$ is based on differences, which means that WLOG $\xi = \epsilon$. So $\delta^{2}$ can be computed from $\epsilon$. First, we compute the arithmetic mean of the entries in $\epsilon$, excluding the second. By \eqref{E:delta.i.2=gamma.sum}, 
    \begin{equation}  \label{E:mu=xi.bar2}
      \mu := \bar{\xi}^{2} = \frac{1}{n-1} ( \epsilon_{1} + \gamma \cdot 1_{n-2} )
        = \frac{1}{n-1} ( \epsilon_{1} - \epsilon_{2} ). 
    \end{equation}
By \eqref{E:delta.i.defn.2}, $\delta^{2} = \epsilon - 1_{n} \mu$. 

Thus, by \eqref{E:dist1.squared} and \eqref{E:delta.i.2=gamma.sum}, the squared distance, $dst_{2}^{2} := dist(x, \Pf_{12})^{2}$, from $\Pf_{12}$ to $x$ is
    \begin{align*}
      dst_{2}^{2} &= (\epsilon_{1} - \mu)^{2} + |\gamma - \mu 1_{n-2} |^{2} \\
        &= \epsilon_{1}^{2} - 2 \epsilon_{1} \mu 
          + \mu^{2} + |\gamma|^{2} - 2 (\gamma \cdot 1_{n-2}) \mu 
            + |1_{n-2}|^{2} \mu^{2} \\
         &= \epsilon_{1}^{2} - 2 \epsilon_{1} \mu + \mu^{2}
          + |\gamma |^{2} + 2 \epsilon_{2} \mu + (n-2) \mu^{2} \\
         &= \epsilon_{1}^{2} - 2 \epsilon_{1} \mu
          + |\gamma |^{2} + 2 \epsilon_{2} \mu + (n-1) \mu^{2} .
    \end{align*}
Hence, by \eqref{E:non-i.deltas.sum.to.0}, \eqref{E:dist1.squared}, and \eqref{E:mu=xi.bar2},
    \begin{align}  \label{E:further.expansion.of.dist2.sqrd}
      dst_{2}^{2} &= \epsilon_{1}^{2} - 2 (\epsilon_{1} - \epsilon_{2}) \mu
        + (dst_{1}^{2} - \epsilon_{2}^{2}) + (n-1) \mu^{2} \notag \\
          &= dst_{1}^{2} - \epsilon_{2}^{2} + \epsilon_{1}^{2}
            - 2 (\epsilon_{1} - \epsilon_{2}) \frac{1}{n-1} (\epsilon_{1} - \epsilon_{2})
              + (n-1) \left( \frac{1}{n-1} (\epsilon_{1} - \epsilon_{2} \right)^{2} \\
         &= dst_{1}^{2} - \epsilon_{2}^{2} + \epsilon_{1}^{2}
            - 2 \frac{1}{n-1} (\epsilon_{1} - \epsilon_{2})^{2}
              + \frac{1}{n-1} (\epsilon_{1} - \epsilon_{2} )^{2} \notag \\
         &= dst_{1}^{2} - \epsilon_{2}^{2} + \epsilon_{1}^{2}
            - \frac{1}{n-1} (\epsilon_{1} - \epsilon_{2})^{2} . \notag 
    \end{align}
Thus, 
    \begin{align*}
      dst_{2}^{2} - dst_{1}^{2} &= \epsilon_{1}^{2} - \epsilon_{2}^{2} 
        - \frac{1}{n-1} (\epsilon_{1} - \epsilon_{2})^{2} \\
            &= (\epsilon_{1} - \epsilon_{2}) \left( \epsilon_{1} + \epsilon_{2}
             -  \frac{1}{n-1} (\epsilon_{1} - \epsilon_{2}) \right) \\
            &= (\epsilon_{1} - \epsilon_{2})
               \left( \frac{n-2}{n-1} \epsilon_{1} + \frac{n}{n-1} \epsilon_{2} \right) \\
            &= (\epsilon_{1} - \epsilon_{2}) \frac{n}{n-1}
               \left( \frac{n-2}{n} \epsilon_{1} + \epsilon_{2} \right) .
    \end{align*}
The $dst_{2}^{2} - dst_{1}^{2}$ vanishes if and only if $\epsilon_{2} = \epsilon_{1}$ \emph{or} $\epsilon_{2} = -\tfrac{n-2}{n} \epsilon_{1}$. Moreover, the second derivative of $dst_{2}^{2} - dst_{1}^{2}$ w.r.t.\ $\epsilon_{2}$ is strictly negative, \emph{viz.}\ $-2n/(n-1)$. That means the function $\epsilon_{2} \mapsto dst_{2}^{2} - dst_{1}^{2}$ is strictly positive in the interval between $-\tfrac{n-2}{n} \epsilon_{1}$ and $\epsilon_{1}$ and strictly negative outside that interval. Thus, when 
$\delta^{1}_{2} = \epsilon_{2}$ is between $\delta^{1}_{1} = \epsilon_{1}$ 
and $-\tfrac{n-2}{n} \delta^{1}_{1}$ we have $dst_{2}^{2} \geq dst_{1}^{2}$ with strict inequality if and only if $\delta^{1}_{2}$ lies strictly between. 
  \end{proof}

Recall \eqref{E:r.defn}. Define
    \begin{equation}  \label{E:possible.g(st)}
      g(u) =
        \begin{cases}
          \frac{n-2}{n}, &\text{ if } u \in [0, r], \\
          1 - 2 \frac{\pi - t}{n(\pi-r)} , &\text{ if } u \in (r, \pi] .
        \end{cases}
    \end{equation}
Thus, we have, 
    \begin{equation*} 
      g : [0, \pi] \to \bigl[ (n-2)/n, 1 \bigr] \text{ but } g(u) < 1 \text{ if } u < \pi 
    \end{equation*}
and $g$ is continuous, \emph{non-decreasing}, and piece-wise differentiable, in fact differentiable except at one point. Moreover,   
    \begin{equation*} 
      g(t) = \frac{n-2}{n} \text{ for } t \in [0, r] \text{ and } g(\pi) = 1 . 
    \end{equation*}

Recall the definition, \eqref{E:fi.defn.2}, of $f_{i}$. For $i \in \NN_{n}$ and 
$|t| \in (0, \pi)$, let $s = sign(t)$. Let $\Theta_{\mcl{X}_{i,sign(t)},t} := \Theta_{i,t}$ be the set of $\theta = (\theta_{1}, \ldots, \theta_{n}) \in [-\pi, \pi]^{n}$ with the properties
    \begin{align}  \label{E:big.Theta.Xit.requirements}
          \theta \cdot f_{i} &= 0, \notag \\ 
          -g(st) \pi \leq s \theta_{j} &\leq \pi \quad (j \neq i), \\
          \theta_{i} &= s \pi , \text{ and }  \notag \\
          \text{ there exists } j \neq i \text{ s.t.\ } \theta_{j} &= -s g(st) \pi
            \text{ \emph{or} } \theta_{j} = s \pi .   \notag 
    \end{align}
(Do not confuse this $\Theta$ with the one in theorem \ref{T:Phi.star.Hr.contains.Theta.star.Hr}.) Notice that for no $k \in \NN_{n}$ is $e_{k}$ is proportional to a point in $\Theta_{i,t}$. $\deli$ as in \eqref{E:delta.i.from.xi} will be proportional to $\theta \in \Theta_{i,t}$.

Let $i \in \NN_{n}$ and $s = \pm1$. Let 
    \begin{equation}  \label{E:Uis.defn}
       U_{i,s} := \bigl\{ t e_{i} + \lambda (\theta - t e_{i}) \in \RR^{n} :  
         st \in (0,\pi), \lambda \in [0, r/\pi), 
           \text{ and } \theta \in \Theta_{i,t} \bigr\} .
    \end{equation}
Let $\phi_{2} \in \RR$. Then $U_{i,s} + \phi_{2} 1_{n}$ is a cone with vertex 
at $t e_{i} + \phi_{2} 1_{n} \in \OneNbSpace_{i}$. 
(Let $y = \wrap(t e_{i} + \phi_{2} 1_{n})$. Then $y \in \Pf_{1i}$ and, in the language of \eqref{E:tw.in.cone}, $\{y\} \times (U_{i,s} + \phi_{2} 1_{n})$ is a cone at $y$.) We have in mind starting with $\xi \in \RR^{n}$ satisfying \eqref{E:xk.near.xi.i.bar} and taking 
$\phi_{2} = \bar{\xi}^{i}$. 

Let $\omega \in U_{i,s}$. Then, by \eqref{E:big.Theta.Xit.requirements}, 
$\omega \cdot f_{i} = 0$. Let $\phi_{2} \in \RR$, and let 
$\xi = \omega + \phi_{2} 1_{n}$. Then, by \eqref{E:delta.i.defn.2}, $\omega = \deli(\xi)$. This justifies denoting points of $U_{i,s}$ by $\deli$. In particular, 
\eqref{E:non-i.deltas.sum.to.0} holds for $\deli = \omega$. 

Let 
    \begin{gather}  \label{E:Uis.params}
      t \in s \, (0,\pi) = 
          \begin{cases}
            (0, \pi) , &\text{ if } s = +1 , \\
            (-\pi,0) , &\text{ if } s = -1 ;
          \end{cases} \\
        \lambda \in [0, r/\pi); \text{ and } \theta \in \Theta_{i,t} . \notag 
    \end{gather} 
These are precisely the requirements laid out in \eqref{E:Uis.defn}. Note that, by \eqref{E:r.defn}, 
    \begin{equation}  \label{E:lambda<1/4}
      \lambda < 1/4.
    \end{equation} 
Let $t$, $\theta$, and $\lambda$ be as in \eqref{E:Uis.params}. Let 
    \begin{equation}  \label{E:delta.i.defined.in.trms.of.t.lambda.theta}
      \deli := t e_{i} + \lambda (\theta - t e_{i}) 
        = (1-\lambda) t e_{i} + \lambda \theta .
    \end{equation} 
Then, by \eqref{E:Uis.params}, $\deli \in U_{i,s}$. We have  
    \begin{align}  \label{E:delta.i.and.delta.k.eqns}
      \deli_{i} &= t + \lambda(\theta_{i} - t) = t + \lambda(s \pi - t)
        = (1-\lambda) t + \lambda s \pi \text{ and } \\
      \deli_{k} &= \lambda \theta_{k} \text{ for } k \neq i \notag .
    \end{align}
By \eqref{E:big.Theta.Xit.requirements}, 
    \begin{equation}  \label{E:delta.dot.fi=0}
      \deli \cdot f_{i} = 0 . 
    \end{equation}

An easy consequence of the preceding together with \eqref{E:possible.g(st)}, \eqref{E:Uis.params}, \eqref{E:big.Theta.Xit.requirements}, and \eqref{E:Uis.defn},  is:
    \begin{align}  \label{E:|delta.k|<r}
      s \deli_{i} &\in (0, \pi) , \notag \\
      -r \leq - r g(st) < - \lambda g(st) \pi \leq \lambda s \theta_{k} &= s \deli_{k} 
        = \lambda s \theta_{k} \leq \lambda \pi < r, \quad \text{ if } k \neq i ,  \\
        | \deli_{k} | &< | \deli_{i} | , \quad \text{ if } k \neq i \notag .
    \end{align}
As we noted above, if $\phi \in \RR$ and $\xi = \deli + \phi 1_{n}$, then 
$\bar{\xi}^{i}= \phi$ and $\deli(\xi)$ as defined in \eqref{E:delta.i.defn.2} equals 
$\deli$ as defined in \eqref{E:delta.i.defined.in.trms.of.t.lambda.theta}. 
By \eqref{E:|delta.k|<r}, 
$-r < \deli_{k} < r$ for $k \neq i$. Therefore, reassuringly, \eqref{E:xk.near.xi.i.bar} holds for $\xi$. 

Here are some useful conditions on $\deli$:
    \begin{multline}  \label{E:condns.on.delta.i}
      \sum_{k \neq i} \deli_{k} = 0 ,  \; 
         s \deli_{i} \in (0, \pi) , \text{ and } \deli_{k} \in ( - r, r)  \text{ if } k \neq i . \\  
          \text{Moreover, for } j \neq i \text{ we have } 
            \;-\frac{n-2}{n} s \deli_{i} < s \deli_{j} < s \deli_{i} . 
    \end{multline}

Let $\deli$ be defined by \eqref{E:delta.i.defined.in.trms.of.t.lambda.theta}, where  \emph{Claim:} 
    \begin{equation}  \label{E:when.delta.is.nice}
      \text{Suppose $t$, $\theta$, and $\lambda$ satisfy \eqref{E:Uis.params}. Then \eqref{E:condns.on.delta.i} holds. } 
    \end{equation}
So \eqref{E:delta.(n-2)/n.ineq} holds strictly for $\deli$. The first two assertions are immediate from \eqref{E:delta.dot.fi=0} and \eqref{E:|delta.k|<r}. As for the last assertion, let $j \neq i$. We already know from \eqref{E:|delta.k|<r} again that
    \begin{equation*}
      s \deli_{k} < s \deli_{i} .
    \end{equation*}
I.e., the right half of the inequality at the end of \eqref{E:condns.on.delta.i} is verified.

It remains to verify the left half. WLOG $s = +1$. First let $t \in (0,r]$. Then, by \eqref{E:possible.g(st)}, $g(st) = (n-2)/n$. Then, by \eqref{E:Uis.params},  \eqref{E:big.Theta.Xit.requirements} and \eqref{E:r.defn}, we have 
$(1-\lambda) t > 0$. Thus, 
    \begin{equation*}
      -\frac{n-2}{n} \deli_{i} 
        = -\frac{n-2}{n} \bigl[ (1-\lambda) t + \lambda \pi \bigr]
          <  -\frac{n-2}{n} \lambda \pi = -g(t) \lambda \pi \leq \lambda \deli_{j} .
    \end{equation*}
This verifies the second half when $t \in (0,r]$. 

Continuing to assume $s = +1$ (so, by \eqref{E:Uis.params}, $t > 0$) now let $t \in (r, \pi)$. By \eqref{E:delta.i.and.delta.k.eqns} and \eqref{E:big.Theta.Xit.requirements}, 
    \begin{equation*}
      -\lambda g(st) \pi \leq s \lambda \theta_{j} = s \deli_{j} .
    \end{equation*}
Thus, it suffices to show $-\frac{n-2}{n} \deli_{i} < -g(t) \lambda \pi$. Suppose not. Then 
    \begin{equation}  \label{E:delta.ineq.fails}
      -\frac{n-2}{n} \deli_{i} 
            = -\frac{n-2}{n} \bigl[ (1-\lambda) t + \lambda \pi \bigr] 
              \geq -g(t) \lambda \pi .
    \end{equation}
Then, by \eqref{E:possible.g(st)}, 
    \begin{align}  \label{E:when.(n-2)/n.etc.geq.-g(t)}
      \frac{n-2}{n} \bigl[ (1-\lambda) t + \lambda \pi \bigr] 
         &\leq \lambda \left[ 1 - 2 \frac{\pi - t}{n(\pi-r)} \right] \pi \notag \\
      \Leftrightarrow (n-2) \bigl[ (1-\lambda) t + \lambda \pi \bigr]
        &\leq \lambda \left[ n - 2 \frac{\pi - t}{\pi-r} \right] \notag \pi \\
      \Leftrightarrow (n-2) (1-\lambda) t + (n-2) \lambda \pi 
        &\leq \lambda n \pi - 2 \lambda \frac{\pi - t}{\pi-r} \pi \\
      \Leftrightarrow (n-2) (1-\lambda) t + 2 \lambda \frac{\pi - t}{\pi-r} \pi  
        &\leq 2 \lambda \pi \notag \\
      \Leftrightarrow (n-2) (1-\lambda) t + 2 \lambda \frac{\pi - t}{\pi-r} \pi 
        - 2 \lambda \pi &\leq 0 \notag .
    \end{align} 
By \eqref{E:n.nvar.k.sizes}, this is false if $\lambda = 0$. So assume $\lambda > 0$. 

For $\lambda$, $r$ in the closures of their appropriate ranges (specified by \eqref{E:Uis.params} and \eqref{E:r.defn}) and $t \in [r, \pi]$, let $f(t,\lambda,r)$ denote the LHS of the last of the preceding inequalities:
    \begin{equation*}
      f(t,\lambda,r) := (n-2) (1-\lambda) t 
        + 2 \lambda \frac{\pi - t}{\pi-r} \pi - 2 \lambda \pi .
    \end{equation*}
\eqref{E:when.(n-2)/n.etc.geq.-g(t)} tells us that 
    \begin{equation*}
      f(t,\lambda,r) \leq 0 . 
    \end{equation*}

Since $t \geq r$, $2 \pi (\pi-t)/(\pi-r) - 2 \pi \leq 0$. It follows that  
$\tfrac{\partial}{\partial \lambda} f(t,\lambda,r) < 0$. Therefore, we make $f(t,\lambda,r)$ smaller by replacing $\lambda$ by the upper limit of its domain, \emph{viz.}, by \eqref{E:Uis.params}, $r/\pi$:
    \begin{align*}
      0 \geq f(t,r/\pi,r) &= (n-2) (1-r/\pi) t 
        + 2 (r/\pi) \frac{\pi - t}{\pi-r} \pi - 2 (r/\pi) \pi \\
      \Leftrightarrow 0 &\geq (n-2) (\pi-r) t + 2 r \frac{\pi - t}{\pi-r} \pi - 2 r \pi \\
      \Leftrightarrow 0 &\geq (n-2) (\pi-r) t 
        - 2 r \pi \left( -\frac{\pi - t}{\pi-r} + 1 \right) \\ 
      \Leftrightarrow 0 &\geq (n-2) (\pi-r)^{2} t - 2 r \pi (t-r) \\
    \end{align*}

The derivative w.r.t.\ $t$ of the final expression in the preceding is
$(n-2) (\pi-r)^{2} - 2 r \pi$. This is minimized by taking $r$ to be its maximum allowed value, which by \eqref{E:r.defn} is $\pi/4$, yielding: 
$(n-2) (9/16) \pi^{2} - \pi^{2}/2$. By \eqref{E:n.nvar.k.sizes} this is strictly positive. Currently we are assuming $t \in (r, \pi)$. 

Therefore, we can make 
$(n-2) (\pi-r)^{2} t - 2 r \pi (t-r)$ even smaller by taking $t = r$. This leaves us with 
    \begin{equation*}
      0 \geq (n-2) (\pi-r)^{2} r > 0 .
    \end{equation*}
Contradiction. We conclude that 
$-\frac{n-2}{n} \deli_{i} <  -g(t) \lambda \pi \leq s \lambda \theta_{j} = s \deli_{j}$. The claim \eqref{E:condns.on.delta.i} is now completely proved:
    \begin{equation}  \label{E:if.Uis.then.delta.condns}
      \text{Given } i \in \NN_{n} \text{ and } s = \pm1, \text{ if } 
        \deli \in U_{i,s} \text{ then \eqref{E:condns.on.delta.i} holds.}
    \end{equation} 

\emph{Now we go in the other direction.} 
Let $\deli \in \RR^{n}$ satisfy \eqref{E:condns.on.delta.i}. We find $s = \pm1$, $t$, 
$\lambda$, and $\theta$ satisfying \eqref{E:Uis.params} so that \eqref{E:delta.i.defined.in.trms.of.t.lambda.theta} holds, i.e.\ so $\deli \in U_{i,s}$, defined in \eqref{E:Uis.defn}. 

Since $\deli$ satisfies \eqref{E:condns.on.delta.i} we must have
    \begin{equation}  \label{E:delta.ii.neq.0}
      \deli_{i} \neq 0 . 
    \end{equation}
 
$\deli_{i} \neq 0$. Let 
    \begin{equation}  \label{E:s=sign(delta.ii)}
      s := sign(\deli_{i}) .
    \end{equation}
($sign$ is defined in \eqref{E:sign.function}.)

Let
    \begin{equation}  \label{E:delta.plus.minus}
      \deli_{-} = \min_{k \neq i} s \deli_{k} \text{ and }
        \deli_{+} = \max_{k \neq i} s \deli_{k} .
    \end{equation} 
We have 
    \begin{equation}  \label{E:|delta+.minus.delta+|}
      | \deli_{-} - \delta'^{i}_{-} | < | \deli - \delta'^{i} | 
        \text{ and } | \deli_{+} - \delta'^{i}_{+} | < | \deli - \delta'^{i} | ,
          \qquad \deli , \delta'^{i} \in \RR^{n} .
    \end{equation}
To see this we may take $s = +1$. Let $j, j' \neq i$ satisfy $\deli_{-} = \deli_{j}$ and 
$\delta'^{i}_{-} = \delta'^{i}_{j'}$. Let $\Delta := | \deli - \delta'^{i} |$. Then 
$\deli_{-} = \deli_{j} \geq \delta'^{i}_{j'} - \Delta \geq \delta'^{i}_{-} - \Delta$. Etc. 

By \eqref{E:condns.on.delta.i}, $\sum_{k \neq i} \deli_{k} = 0$ and 
$\deli_{k} \in ( - r, r)$. Hence,
    \begin{equation} \label{E:delta-.leq.delta+}
      -r < \deli_{-} \leq 0 \leq \deli_{+} < r. 
    \end{equation}
We want \eqref{E:delta.i.and.delta.k.eqns} to hold with $\theta$ satisfying  \eqref{E:big.Theta.Xit.requirements}. So we want $t$ and $\lambda$ to satisfy
    \begin{equation} \label{E:s.delta.pm.and}
      - g(st) r < - g(st) \; \lambda \pi \leq \deli_{-} \text{ and } \; 
          r > \lambda \pi \geq \deli_{+} \geq 0 .  .
    \end{equation}

By \eqref{E:possible.g(st)}, $g > 0$. By definition \eqref{E:big.Theta.Xit.requirements} of $\Theta_{i,t}$, \eqref{E:delta.i.and.delta.k.eqns}, and \eqref{E:Uis.params}, we need $t$ and $\lambda$ to satisfy
    \begin{multline} \label{E:lambda.r.constraints}
      \text{first, }0 \leq \lambda < r/\pi \text{ and second, } \\
      -r < -g(st) r < - g(st) \; \lambda \pi =  \deli_{-} \; \text{ and/or } \; 
         \lambda \pi = \deli_{+} \in [0,r) .
    \end{multline}
\eqref{E:possible.g(st)} tells us that by varying $st$ through 
$(0,\pi)$, $g(st)$ varies through $\bigl[ (n-2)/n,1 \bigr)$. Allowing that and varying 
$\lambda \in [0,r/\pi)$ we have that $- g(st) \; \lambda \pi$ varies through $(-r,0]$. 
So permissible manipulation of $t$ and $\lambda$ makes $- g(st) \; \lambda \pi$ cover the range of $\deli_{-}$ permitted by \eqref{E:condns.on.delta.i}.  
Let $\lambda_{-}$ be the solution to the first of the preceding equations 
$- g(st) \; \lambda \pi =  \deli_{-}$ (in terms of $g(st)$, i.e., in terms of $t$) and 
$\lambda_{+}$ the solution to the second $\lambda \pi = \deli_{+}$. Thus, by  
\eqref{E:r.defn}, 
    \begin{multline}  \label{E:lambda-.and.lambda+}
      \lambda_{-} := \lambda_{-}(\deli) := -\frac{1}{g(st) \pi} \deli_{-} \in [0, r/\pi) 
        \subset [0,1) \; \text{ and } \\ 
          \lambda_{+} := \lambda_{+}(\deli) := \frac{1}{\pi} \deli_{+} \in [0, r/\pi)
            \subset [0,1) .
    \end{multline}
Let 
    \begin{equation}  \label{E:lambda(delta).defn}
      \lambda(\deli) := \max \{ \lambda_{-} , \lambda_{+} \} \in [0, r/\pi) .
    \end{equation}
(This is consistent with the requirement on $\lambda$ in \eqref{E:Uis.params}.)

Now we solve for $t$. From \eqref{E:delta.i.and.delta.k.eqns} we get, 
    \begin{equation}  \label{E:generic.eqn.for.t}
      t = \frac{\deli_{i} - \lambda s \pi}{1-\lambda} .
    \end{equation}
(By \eqref{E:lambda-.and.lambda+}, $\lambda < 1/4$ so this is well-defined.) We can quickly dispose of \emph{claim:}
    \begin{equation}  \label{E:st<s.delta.i}
      \text{If } \lambda = 0 \text{ then } st = s \deli_{i} . 
        \text{ Otherwise } st < s \deli_{i}  . 
    \end{equation} 
The $\lambda = 0$ is trivial. Suppose $\lambda > 0$
$\lambda > 0$ but $st \geq s \deli_{i}$. Then, by \eqref{E:generic.eqn.for.t}, 
    \begin{equation*}
      s \deli_{i} - \lambda \pi = s t - \lambda s t \geq s \deli_{i} - \lambda s t 
      \text{ which implies } - \lambda \pi \geq - \lambda s t .
    \end{equation*} 
I.e., $s t \geq \pi$. This contradicts \eqref{E:Uis.params} and proves the claim.

In \eqref{E:generic.eqn.for.t}, we expressed $t$ in terms of $\deli$, $\lambda$, and $s$. Now we take $\lambda = \lambda(\deli)$ and express $t$ solely in terms 
of $\deli$ and $s$. Now, $s := sign(\deli_{i})$, so we will actually express $t$ in terms 
of $\deli$. The solution must satisfy \eqref{E:delta.i.and.delta.k.eqns} and, by \eqref{E:Uis.params}, we must have $st \in (0,\pi)$. The case 
$\lambda(\deli) = \lambda_{+}$ is trivial (See \eqref{E:t.in.trms.of.delta.i.and.lambda+} below.) The tricky case is when 
$\lambda = \lambda(\deli) = \lambda_{-} = - \tfrac{1}{\pi g(st)} \deli_{-}$. By \eqref{E:delta.i.and.delta.k.eqns} and \eqref{E:lambda-.and.lambda+}, the equation to solve is 
    \begin{equation}  \label{E:delta.i.g(st).eqn.lambda-.case}
      \deli_{i} = t- \frac{1}{\pi g(st)} \deli_{-} \, (s \pi - t) , 
        \quad st \in (0,\pi) .
    \end{equation}
 
\eqref{E:delta.i.g(st).eqn.lambda-.case} is equivalent to
    \begin{equation} \label{E:delta.i.F(t).eqn}
      s \deli_{i} = F(s t, r) := F(s t; r, \deli)  
        := s t - \frac{1}{\pi g(s t)} s \deli_{-} (\pi - s t )  .
    \end{equation}
By \eqref{E:Uis.params}, $0 < st < \pi$. (But, by \eqref{E:possible.g(st)}, \eqref{E:delta.i.F(t).eqn} still makes sense if $st = 0$ or $\pi$.) By \eqref{E:delta-.leq.delta+}, $\deli_{-} \leq 0$. And we know $g(st) > 0$. Therefore, by \eqref{E:st<s.delta.i} and \eqref{E:Uis.params},  
    \begin{equation*}
      0 \leq st \leq s \deli_{i} \leq \pi .
    \end{equation*}

By \eqref{E:delta.i.F(t).eqn}, WLOG we may assume $s = +1$. So \eqref{E:delta.i.F(t).eqn} becomes:
    \begin{equation}  \label{E:delta.i.F(t).eqn.-}  
      \deli_{i} = F(t, r)  
        = t + \frac{\deli_{-} }{\pi g(t)} t - \frac{\deli_{-}}{g(t)} 
             = t - \frac{ \pi - t }{\pi g(t)} \deli_{-} , \quad t \in [0,\pi] . 
    \end{equation} 

The solution to \eqref{E:delta.i.F(t).eqn.-}, if there is one, might not be unique. A necessary and sufficient condition for the solution to be unique is for $F$ to be strictly monotonic 
in $t \in [0,\pi]$. Let $t$ belong to one of the open intervals in which $g$ is differentiable. We have
    \begin{align}  \label{E:partial.F}
      \frac{\partial}{\partial t} F(t,r) 
        &= 1 + \frac{\deli_{-}}{\pi} \; \frac{g(t) - t g'(t)}{g^{2}(t)} 
          + \frac{\deli_{-} g'(t)}{g^{2}(t)} \\
        &= 1 + \frac{\deli_{-}}{\pi g(t)} + \deli_{-} \frac{\pi - t}{\pi g(t)^{2}} 
          g'(t) \notag . 
    \end{align}

By \eqref{E:possible.g(st)}, $g'(t) = 0$ if $0 \leq t < r$. Therefore, by \eqref{E:partial.F}, the assumption that $s = +1$, \eqref{E:condns.on.delta.i}, \eqref{E:n.nvar.k.sizes}, and \eqref{E:r.defn}, we then have
    \begin{equation}  \label{E:partil.F.lwr.bnd.+}
      \frac{\partial}{\partial t} F(t,r) = 1 + \frac{\deli_{-}}{\pi g(t)} 
        = 1 + \frac{s \deli_{-}}{\pi g(t)}
          > 1 - \frac{r g(t)}{\pi g(t)} 
            = 1 - \frac{r}{\pi} \geq \frac{3}{4} , \quad t \in (0,r) .
    \end{equation}

Recall that $s = +1$. Since $g(t) \geq (n-2)/n$ by \eqref{E:possible.g(st)}, if $r < t < \pi$, we have, by \eqref{E:partial.F}, \eqref{E:lambda.r.constraints}, \eqref{E:possible.g(st)}, 
and \eqref{E:n.nvar.k.sizes}, 
    \begin{align*}
      \frac{\partial}{\partial t} F(t,r) 
            &\geq 1 - \frac{r g(t)}{\pi g(t)} 
              - r g(t) \frac{\pi - t}{\pi g(t)^{2}} g'(t) \notag \\
            &= 1 - \frac{r}{\pi} - r \frac{\pi - r}{\pi g(t)} g'(t) \notag \\
            &\geq 1 - \frac{r}{\pi} - r \frac{\pi - r}{\pi} \frac{n}{n-2} g'(t) \notag \\
            &\geq 1 - \frac{r}{\pi} - r \frac{\pi - r}{\pi} \times 3 g'(t), \notag \\ 
            &\qquad r < t < \pi .
    \end{align*}
 
By \eqref{E:possible.g(st)}, $g'(t) = \tfrac{2}{n(\pi-r)}$ on $t \in (r, \pi]$ and, by \eqref{E:n.nvar.k.sizes} again and \eqref{E:r.defn},  
    \begin{multline}  \label{E:partil.F.lwr.bnd.-}
      \frac{\partial}{\partial t} F(t,r) > 1 - \frac{r}{\pi} 
        - 3 r \frac{\pi - r}{\pi} \frac{2}{n(\pi-r)}
          = 1 - \frac{r}{\pi} - 3 r \frac{2}{n \pi} \\
            \geq 1 - \frac{r}{\pi} - 3 r \frac{2}{3 \pi}
              = 1 - \frac{3 r}{\pi} \geq 1 - \frac{3}{4} = \frac{1}{4} > 0  \text{ for } 
                t \in (r, \pi) .
    \end{multline}

Combining this with \eqref{E:partil.F.lwr.bnd.+}, we have $\frac{\partial}{\partial t} F(t,r) > 0$ ($t \neq r$) as desired. Thus,
    \begin{equation}  \label{E:F(t,r).increasing.in.t}
      F(s t,r) \text{ is strictly increasing in } s t \in [0,\pi] . 
    \end{equation}
 
Assume $s = +1$. So \eqref{E:delta.i.F(t).eqn.-} has a unique solution in $t$ providing 
$\deli_{i} \in \bigl[ F(0,r), F(\pi,r) \bigr]$. We compute $F$ on the endpoints of the interval $[0,\pi]$. By \eqref{E:delta.i.F(t).eqn.-}, \eqref{E:possible.g(st)}, and \eqref{E:delta-.leq.delta+}: 
    \begin{equation*}
      F(0,r) =  -\frac{n \pi}{\pi (n-2)} \deli_{-} = -\frac{n}{n-2} \deli_{-} \geq 0 
    \end{equation*}
and $F(\pi,r) =  \pi$. By \eqref{E:condns.on.delta.i}, the image $[0,\pi]$ of $F$ includes the range of $s \deli_{i}$. Therefore, the equation \eqref{E:delta.i.F(t).eqn} has a unique solution $t = t_{-} := t_{-}(\deli)$ in the domain, $[0,\pi]$, of $F$. But by \eqref{E:delta.ii.neq.0}, $0 < |\deli_{i}| < \pi$, so actually
    \begin{equation}  \label{E:st-.in.(0,pi)}
      s \, t_{-}(\deli) \in (0, \pi) . 
    \end{equation}   

Recall \eqref{E:delta.i.F(t).eqn.-}, \eqref{E:partil.F.lwr.bnd.-}, and \eqref{E:partil.F.lwr.bnd.+}.  
Let $\delta'^{i} \in \RR^{n}$ be another vector for which 
$\lambda(\deli) = \lambda_{-}$. Then, 
if $\lambda(\deli) = \lambda_{-}(\deli)$ and 
$\lambda(\delta'^{i}) = \lambda_{-}(\delta'^{i})$,   
    \begin{equation*}  \label{E:t(delta).Lip}
     |\delta'^{i} - \deli | \geq |\delta'^{i}_{i} - \deli_{i} | 
       = \Bigl| F \bigl( t(\delta'^{i}) \bigr) 
         - F \bigl( t(\deli) \bigr) \bigr| \geq \frac{1}{2} \bigl| t(\delta'^{i}) - t(\deli) \bigr| .
    \end{equation*}
Therefore, 
    \begin{equation}  \label{E:t-(deli).Lip}
      t_{-} \text{ is Lipschitz in } \deli \text{ with Lipschitz constant } L_{-} := 2 . 
    \end{equation}

Recall \eqref{E:delta.plus.minus} and \eqref{E:generic.eqn.for.t}. Now suppose 
$\lambda(\deli) = \lambda_{+} = \tfrac{1}{\pi} s \deli_{+}$.  Then, by \eqref{E:generic.eqn.for.t} and \eqref{E:lambda-.and.lambda+}, 
    \begin{equation}  \label{E:t.in.trms.of.delta.i.and.lambda+}
      t_{+}(\deli) =  \frac{\deli_{i} - \tfrac{1}{\pi} \deli_{+} s \pi}
          {1-\tfrac{1}{\pi} s \deli_{+}}
            = \frac{ \deli_{i} - s \deli_{+} }{\pi - \deli_{+}} \pi \\
              = s \frac{ s \deli_{i} - \deli_{+} }{\pi - \deli_{+}} \pi . 
    \end{equation}
By \eqref{E:condns.on.delta.i} and \eqref{E:s=sign(delta.ii)}, 
$\deli_{+} < s \deli_{i} < \pi$. Therefore, 
    \begin{equation}  \label{E:st+.in.(0,pi)}
      s t_{+} \in (0, \pi) .
    \end{equation}
     
By \eqref{E:comp.of.Lips.is.Lip} and examples \ref{Ex:projection.is.Lip}, \ref{Ex:lattice.operations.are.Lip}, $t_{+}(\deli)$ is locally Lipschitz in $\deli$. Since, 
by \eqref{E:condns.on.delta.i} and \eqref{E:r.defn}, $s \deli_{+} < r \leq \pi/4$, 
    \begin{equation}  \label{E:t+(delta).is.Lip}
      t_{+}(\deli) \text{ is Lipschitz.} 
    \end{equation} 
Let $L_{+} \in (0, \infty)$ be a Lipschitz constant. 

By \eqref{E:possible.g(st)}, $g$ is Lipschitz and bounded away from 0. Therefore, by example \ref{Ex:ratnl.fns.loc.Lip}, \eqref{E:comp.of.Lips.is.Lip}, and \ref{Ex:ratnl.fns.loc.Lip}, we have that $1/g(t)$ is Lipschitz in $t \in [0,\pi]$. It then follows from \eqref{E:lambda-.and.lambda+}, \eqref{E:t-(deli).Lip}, and \eqref{E:t+(delta).is.Lip} that 
    \begin{equation}  \label{E:lambda-.and.lambda+.are.each.Lip}
      \lambda_{-}(\deli) \text{ and } \lambda_{+}(\deli) 
        \text{ are each Lipschitz}. 
    \end{equation}

Recall \eqref{E:lambda(delta).defn}. Define: 
    \begin{equation*}
      t(\deli) = t_{-} \text{ or } t_{+} 
        \text{ according as $\lambda(\deli)$ is $\lambda_{-}$ or } \lambda_{+} .
    \end{equation*} 
Then, by \eqref{E:st-.in.(0,pi)} and \eqref{E:st+.in.(0,pi)}, 
    \begin{equation}  \label{E:st.in.(0,pi)}
      s \, t(\deli) \in (0, \pi) . 
    \end{equation}

Recall \eqref{E:lambda(delta).defn}. We now improve upon \eqref{E:lambda-.and.lambda+.are.each.Lip} by showing that in fact $\lambda(\deli)$ is Lipschitz in $\deli$. Let $s = \pm 1$, $\deli , \delta'^{i} \in \RR^{n}$.  
Let $\tau := \tau(\deli) :=  1/\bigl( g \bigl[ s t_{-}(\deli) \bigr] \pi \bigr)$ and 
$\tau' := \tau'(\delta'^{i}) :=  1/\bigl( g s \bigl[ t_{-}(\delta'^{i}) \bigr] \pi \bigr)$. (Thus, we use the same $s$ for $\deli$ and $ \delta'^{i}$. WLOG $s = +1$.) By \eqref{E:possible.g(st)}, $\tau, \tau' \leq n/(n-2) \pi$.  Then, by \eqref{E:t-(deli).Lip},  \eqref{E:t+(delta).is.Lip}, and the fact that $1/g$ is Lipschitz, there exists 
$K$ s.t.\ 
    \begin{equation*}
      \bigl| \tau \deli_{-} - \tau' \delta'^{i}_{-} \bigr| \leq \Delta 
        := K |\deli - \delta'^{i}| .
    \end{equation*}
Consider the following scenario: 
$\lambda(\deli) = \lambda_{-}(\deli)$ and $\deli_{+} = \deli_{j} \geq 0$. 
(But $\delta'^{i}_{j} < 0$ is possible.)
So $\lambda_{+}(\deli) = \deli_{j}/\pi \leq \lambda(\deli)$. 
Suppose also $\lambda(\delta'^{i}) = \lambda_{+}(\delta'^{i})$ and 
$\delta'_{-} = \delta'^{i}_{k} \leq 0$ so 
$\lambda_{-}(\delta'^{i}) = -\delta'^{i}_{k} \tau \leq \lambda(\delta'^{i})$. 
($\deli_{k} > 0$ is possible.) Then
    \begin{equation*}
      - \lambda(\delta'^{i}) \leq - \lambda_{-}(\delta'^{i}) 
        = \delta'^{i}_{k} \tau' \geq \deli_{k} \tau - \Delta \geq 
          - \lambda_{-}(\deli) - \Delta = -\lambda(\deli) - \Delta .
    \end{equation*}
Thus, we have $K |\deli - \delta'^{i}| \geq \lambda(\delta'^{i}) - \lambda(\deli)$. Similarly, 
    \begin{equation*}
      - \lambda(\deli) \leq - \lambda_{+}(\deli) 
        = -\deli_{j}/\pi \geq -\delta'^{i}_{j}/\pi - \Delta \geq 
          - \lambda_{+}(\delta'^{i}) - \Delta = -\lambda(\delta'^{i}) - \Delta .
    \end{equation*}
Thus, we also have 
$\Delta \geq  \lambda(\deli) - \lambda(\delta'^{i})$. 
I.e., $\bigl| \lambda(\delta'^{i}) - \lambda(\deli) \bigr| \leq \Delta$. Combining this with \eqref{E:lambda-.and.lambda+.are.each.Lip}, we get 
    \begin{equation}  \label{E:lambda.is.Lip}
      \lambda(\deli) \text{ is Lipschitz in } \deli . 
    \end{equation}

It follows from \eqref{E:generic.eqn.for.t} that
    \begin{equation}  \label{E:t.is.Lip}
      t(\deli)  \text{ is Lipschitz in } \deli .
    \end{equation}

Now we solve \eqref{E:delta.i.and.delta.k.eqns} for $\theta$ subject to \eqref{E:lambda.r.constraints}. By \eqref{E:delta.i.defined.in.trms.of.t.lambda.theta} we must have
    \begin{equation}  \label{E:theta(delta.i).defn}
      \theta(\deli) := 
        \begin{cases}
              \lambda(\deli)^{-1} \Bigl[ \deli - \bigl( 1-\lambda(\deli) \bigr) t(\deli) e_{i} \Bigr] , 
                &\text{ if } \lambda(\deli) \neq 0 , \\
             \text{arbitrary element of } \Theta_{s, t(\deli)}, &\text{ otherwise}  .
        \end{cases}  
    \end{equation}
By \eqref{E:lambda.is.Lip} and \eqref{E:t.is.Lip}, 
    \begin{equation}  \label{E:lambda.theta.is.Lip}
      \lambda(\deli) \theta(\deli) \text{ is Lipschitz in } \deli .
    \end{equation} 

We prove that $\theta(\deli) \in \Theta_{s, t(\deli)}$. This is necessary only in the case 
$\lambda(\deli) > 0$. So assume $\lambda(\deli) > 0$. By \eqref{E:condns.on.delta.i} and \eqref{E:fi.defn.2}, we have $\theta(\deli) \cdot f_{i} = 0$, consistent with \eqref{E:big.Theta.Xit.requirements}. It follows from \eqref{E:lambda-.and.lambda+} and \eqref{E:delta-.leq.delta+} that $\deli_{-} < 0$ or $\deli_{+} > 0$. 
But $\theta(\deli) \cdot f_{i} = 0$ then implies  
    \begin{equation}  \label{E:delta-<0<delta+}
      \deli_{-} < 0 < \deli_{+} .
    \end{equation}  
It follows from \eqref{E:lambda-.and.lambda+} that 
    \begin{equation}  \label{E:both.lambdas.poz}
      \lambda_{-}, \lambda_{+} > 0 .
    \end{equation}
In particular, $\lambda_{-}^{-1}, \lambda_{+}^{-1} \in (0, \infty)$. 

We are assuming \eqref{E:condns.on.delta.i} and 
$\lambda(\deli) > 0$. Recall \eqref{E:lambda(delta).defn} and 
write $t(\deli) = t$, $\lambda_{-} = \lambda_{-}(\deli)$, and 
$\lambda_{+} = \lambda_{+}(\deli)$. 
First, assume $\lambda(\deli) = \lambda_{-}$, so by \eqref{E:lambda(delta).defn}, 
$\lambda_{-} \geq \lambda_{+}$. Let $j \neq i$. 
By \eqref{E:delta-<0<delta+}, we have $-g(t) \pi/\deli_{-} > 0$.  
By \eqref{E:delta.plus.minus}, $s \deli_{j} \geq \deli_{-}$. Then \eqref{E:theta(delta.i).defn}, and \eqref{E:lambda-.and.lambda+}, 
    \begin{align}  \label{E:s.theta.j.ineqs.when.lambda=lambda-}
      - g(t) \pi = - \frac{g(t) \pi}{\deli_{-}} \deli_{-} \leq -\frac{g(t) \pi}{\deli_{-}} s \deli_{j}
        = \lambda_{-}^{-1} s \deli_{j} \quad & \notag \\
        = s \, \theta_{j}&(\deli) = \\ 
        & \quad \lambda_{-}^{-1} s \deli_{j} 
          \leq \lambda_{-}^{-1} \deli_{+} \leq \lambda_{+}^{-1} \deli_{+} 
            = \frac{\pi}{\deli_{+}} \deli_{+} = \pi \notag .
    \end{align}
This is consistent with \eqref{E:big.Theta.Xit.requirements}. 

Similarly, suppose 
$\lambda(\deli) = \lambda_{+}$. So $\lambda_{-}^{-1} \geq \lambda_{-}^{-1}$. Now, by \eqref{E:delta-<0<delta+}, $\deli_{-} < 0$. Hence, 
$\lambda_{-}^{-1} \deli_{-} \leq \lambda_{+}^{-1} \deli_{-}$. We have,
    \begin{align}  \label{E:s.theta.j.ineqs.when.lambda=lambda+}
      - g(t) \pi = - \frac{g(t) \pi}{\deli_{-}} \deli_{-} 
        = \lambda_{-}^{-1} \deli_{-} \leq \lambda_{+}^{-1} \deli_{-} 
         \leq \lambda_{+}^{-1} s \deli_{j} \quad & \notag \\
          = s \, \theta_{j}&(\deli) = \\
           & \quad \lambda_{+}^{-1} s \deli_{j} \leq \lambda_{+}^{-1} \deli_{+} 
             = \frac{\pi}{\deli_{+}} \deli_{+} = \pi \notag ,
    \end{align}
again consistent with \eqref{E:big.Theta.Xit.requirements}. 

Write $\lambda = \lambda(\deli)$. By \eqref{E:theta(delta.i).defn} and \eqref{E:generic.eqn.for.t},
    \begin{equation*}
      s \theta_{i}(\deli) = \lambda^{-1} s \bigl[ \deli_{i} - (1-\lambda) t \bigr]
        = \lambda^{-1} s \bigl[ \deli_{i} - (\deli_{i} - \lambda s \pi) \bigr]
          = \lambda^{-1} s (\lambda s \pi) = \pi . 
    \end{equation*}
Again in agreement with \eqref{E:big.Theta.Xit.requirements}. 

Suppose again that $\lambda(\deli) = \lambda_{-} > 0$. Let $j \neq i$ satisfy 
$s \deli_{j} = \deli_{-}$. Then the inequality in the top member of \eqref{E:s.theta.j.ineqs.when.lambda=lambda-} becomes an equality and we get 
$- g(st) \pi = s \, \theta_{j}(\deli)$, consistent with \eqref{E:big.Theta.Xit.requirements}. Similarly, suppose again that $\lambda(\deli) = \lambda_{+}$. Let $j \neq i$ satisfy 
$s \deli_{j} = \deli_{+}$. Then the inequality in the bottom member of \eqref{E:s.theta.j.ineqs.when.lambda=lambda+} becomes equality and we get 
$s \, \theta_{j}(\deli) = \pi$, consistent with \eqref{E:big.Theta.Xit.requirements}.

This completes the proof that, given $\deli \in \RR^{n}$ satisfying \eqref{E:condns.on.delta.i}, we can find ``$s = \pm1$, $t$, $\lambda$, and $\theta$ satisfying \eqref{E:Uis.params} so that \eqref{E:delta.i.defined.in.trms.of.t.lambda.theta} holds, i.e.\ so $\deli \in U_{i,s}$, defined in \eqref{E:Uis.defn}.'' To repeat:
    \begin{equation}  \label{E:delta.conditions.define.U}
      \text{Given } i \in \NN_{n} \text{ and } s = \pm1, \; 
        \deli \in U_{i,s} \text{ if and only if \eqref{E:condns.on.delta.i} holds.}
    \end{equation}

Moreover, by lemma \ref{L:tween.-(n-2)/n.and.1}, \eqref{E:lambda.is.Lip}, \eqref{E:t.is.Lip}, and \eqref{E:lambda.theta.is.Lip}, we have,
  \begin{lemma}  \label{L:lambda.t.theta.uniqueness}
Suppose $\deli \in \RR^{n}$ satisfies \eqref{E:condns.on.delta.i}. Then:
    \begin{itemize}
      \item If $\phi \in \RR$ then $\Pf_{1i}$ is the closest lobe to 
    $\wrap(\deli + \phi 1_{n})$.  \label{I:Pf.1i.closest}
      \item $s := sign(\deli_{i})$, $\lambda = \lambda(\deli)$, $t = t(\deli)$, and 
    $\theta = \theta(\deli)$ are the \emph{unique} solutions to 
    \eqref{E:delta.i.defined.in.trms.of.t.lambda.theta} satisfying \eqref{E:Uis.params}.
      \label{I:unique.solns}
      \item $\lambda(\deli), t(\deli)$, and $\lambda(\deli) \theta(\deli)$ are Lipschitz 
    in $\deli$.  \label{I:Lipschitzosity}
    \end{itemize}  
  \end{lemma}

Let $s = \pm 1$ and $y \in \mcl{X}_{i,s}$. We can write 
$y = \, \wrap( \phi_{1} e_{i} + \phi_{2} 1_{n} )$, with $s \phi_{1} \in (0,\pi)$ 
and $\phi_{2} \in \RR$. Recall \eqref{E:big.Theta.Xit.requirements}.  
    \begin{multline}  \label{E:C[y].y.in.Xis}
      \text{If } y = \, \wrap( t e_{i} + \alpha_{2} 1_{n} ) \in \mcl{X}_{i,s} \; 
        (\text{with } s t \in (0, \pi)) \text{ define } \\ 
      C_{\mcl{X}_{i,s}}[y] = C[y] 
        := \Bigl\{ \bigl( y, \lambda (\theta - t e_{i}) \bigr) \in T_{y} \D : 
          \lambda \in [0, r/\pi), \, \theta \in \Theta_{i,t} \Bigr\} . 
    \end{multline} 
Thus, if $t := \phi_{1}$, then $t$, $\lambda$, and $\theta$ satisfy the requirements spelled out in \eqref{E:Uis.defn} and recapitulated in \eqref{E:Uis.params}. 
By \eqref{E:vector.ops.on.TD} and \eqref{E:tw.in.cone}, $C[y]$ is a cone. In conformity with \eqref{E:boldF.[E].defns}, write 
$C[\mcl{X}_{i,s}] := \bigcup_{y \in \mcl{X}_{i,s}} C[y]$. In accordance with \eqref{E:CL.in.Euc.space} we can write $C[y] =  \msf{CL}$. It is easy to see that the link, $\msf{L}$, of the cone is a compact finite cell complex of dimension $n-3 = d-p-1$ and, hence, by \eqref{E:cell.is.tame.strat.space}, 
    \begin{multline}  \label{E:Xis.cone.has.tame.link}
      \msf{L} \text{ is a compact stratified space, has dimension $n-3 = d-p-1$, } \\
        \text{ and satisfies \eqref{E:tameness.of.stratified.space}.}
    \end{multline}
As required by part \ref{I:L.A.tameness} of definition \ref{D:fibering.by.cones}. (By \eqref{E:n.nvar.k.sizes}, $n=3$ is possible. So in that case $L$ consists of two points, as in figure \ref{F:ConeBundle}.)
 
It seems like $C[\mcl{X}_{i,s}]$ should be related to $U_{i,s}$ defined in \eqref{E:Uis.defn}. They are related. Let $s = \pm 1$, $y \in \mcl{X}_{i,s}$, 
$(y,v) \in C[y]$. By \eqref{E:mclX,mclY.defns}, 
 $y = \, \wrap( \alpha_{1} e_{i} + \alpha_{2} 1_{n} ) \in \mcl{X}_{i,s}$ with 
 $s \alpha_{1} \in (0, \pi)$. We extract from $(y,v)$ an $\alpha_{2} \in \RR$ and a 
 $\deli \in U_{i,s}$. By \eqref{E:alphas.from.y}, we can recover $\alpha_{1} \in s(0,\pi)$ 
 and $\alpha_{2}$ from $y$ in a Lipschitz fashion. Let $t := \alpha_{1}$ so we can recover $t \in s \, (0,\pi)$ from $(y,v)$ in a Lipschitz fashion. By \eqref{E:C[y].y.in.Xis}, 
 $v = \lambda (\theta - \alpha_{1} e_{i})$ for some 
$\lambda \in [0, r/\pi)$ and $\theta \in \Theta_{i,\alpha_{1}}$. So if 
$\deli = t e_{i} + v = \alpha_{1} e_{i} + v$ then $\deli$ satisfies \eqref{E:delta.i.defined.in.trms.of.t.lambda.theta} and $t$, 
$\theta$, and $\lambda$ satisfy \eqref{E:Uis.params}. In other words, by \eqref{E:Uis.defn}, $(y,v) \in C[\mcl{X}_{i,s}]$ determines a unique point, 
$\deli = t e_{i} + v$, of $U_{i,s}$. Therefore, by \eqref{E:when.delta.is.nice}, $\deli$ satisfies \eqref{E:condns.on.delta.i}. Hence, by lemma \ref{L:lambda.t.theta.uniqueness}, we can compute $\lambda$, $t$, and $\lambda \theta$, and hence $\deli$, in a Lipschitz fashion. (We have already determined that $t$ can be computed in a Lipschitz fashion.) By \eqref{E:Exp.psi.2}, we have
    \begin{multline}  \label{E:Exp(C[mclX])}
      Exp \bigl( (y,v) \bigr) = Exp_{y} \bigl( \lambda (\theta - \alpha_{1} e_{i}) \bigr) \\
        = \, \wrap \bigl( \alpha_{1} e_{i} + \alpha_{2} 1_{n} 
          + \lambda (\theta - \alpha_{1} e_{i}) \bigr)
            = \, \wrap(\deli + \alpha_{2} 1_{n}) , \\
              \text{ where }
                y = \, \wrap( \alpha_{1} e_{i} + \alpha_{2} 1_{n} ) \in \mcl{X}_{i,s} .
    \end{multline} 
(This formula will require tweaking. See \eqref{E:Exp(C[mclX]).revised} below.)
We also have seen that there is a Lipschitz, $(\alpha_{2}, \deli)$-valued, map from 
$C[\mcl{X}_{i,s}]$ into $\RR \times U_{i,s}$. 

Conversely, let $\deli \in U_{i,s}$ and $\alpha_{2} \in \RR$. We construct from them a point of $C[\mcl{X}_{i,s}]$. By definition of $U_{i,s}$ there exists $t$, $\lambda$, 
$\theta$ satisfying \eqref{E:Uis.params}. By \eqref{E:when.delta.is.nice} it follows that 
$\deli$ satisfies \eqref{E:condns.on.delta.i}. By lemma \ref{L:lambda.t.theta.uniqueness}, this in turn implies we can uniquely, and in a Lipschitz fashion, express $\deli$ in the form 
$t e_{i} + \lambda (\theta - t e_{i})$, with $s := sign(\deli_{i})$, $\lambda$, $t$, 
and $\theta$ as in \eqref{E:Uis.params}. Let $\alpha_{1} := t$ and 
$y := \, \wrap( \alpha_{1} e_{i} + \alpha_{2} 1_{n} )$. 
Then, by \eqref{E:mclX,mclY.defns} and \eqref{E:Uis.params}, and \eqref{E:C[y].y.in.Xis}, $y \in \mcl{X}_{i,s}$ and $\bigl( y, \lambda (\theta - \alpha_{1} e_{i}) \bigr) \in C[y]$. 

We have established that $C[\mcl{X}_{i,s}]$ and $\RR \times U_{i,s}$ are bi-Lipschitz homeomorphic. 

\eqref{E:condns.on.delta.i} can be reified (properties turned into a thing) as follows. 
Let $\mcl{C}_{\mcl{X}_{i,s}} := \mcl{C}_{i,s} \subset \D$ be the set of $x \in \D$ s.t.:
    \begin{align}  \label{E:mclC.Xis.defn}
      x &\in \mcl{G}_{i} , \notag \\
      \deli_{i} &:= \angle_{s}(x_{i}, \bar{x}^{i}) \in (0, \pi) , \\
      \deli_{j} &\in (-r,r) \; \text{ if } j \neq i , \text{ and } \notag \\
      -\frac{n-2}{n} s \deli_{i} &< s \deli_{j}(x) < s \deli_{i} \; \text{ if } j \neq i . \notag
    \end{align}
For $x \in \mcl{C}_{i,s}$ let 
$\mcl{F}(x) := \bigl\{ z \in \mcl{C}_{i,s} : \angle_{s}(\bar{x}^{i}, z_{i}) < \pi/4 \bigr\}$. Then, by \eqref{E:directed.angle.is.Lip} and \eqref{E:locally.Lip.xbari.on.Gi},  
there exists $L < \infty$ s.t.\ $\bigl| \deli(z) - \deli(x) \bigr| < L \, \rho(z,x)$ for every 
$z \in \mcl{E}(x) \cap \mcl{F}(x)$. 
In general, for $x,z \in \mcl{C}_{i,s}$, we have 
$\bigl| \deli(z) - \deli(x) \bigr| \leq \sqrt{n} \pi$. Moreover, if 
$z \in \mcl{C}_{i,s} \setminus \bigl[ \mcl{E}(x) \cap \mcl{F}(x) \bigr]$ then 
$\rho(z,x) \geq \pi/4$. So if 
$z \in \mcl{C}_{i,s} \setminus \bigl[ \mcl{E}(x) \cap \mcl{F}(x) \bigr]$ then
$\bigl| \deli(z) - \deli(x) \bigr| \leq \bigl[ \sqrt{n} \pi/(\pi/4) \bigr] \rho(z,x)$. 
Let $L' := \max(L, 4 \sqrt{n})$. It follows that, if 
$x,z \in \mcl{C}_{i,s}$ then $\bigl| \deli(z) - \deli(x) \bigr| \leq L' \, \rho(z,x)$. This proves,  
    \begin{equation}  \label{E:Lip.delta.on.Gi}   
      \text{The map } \deli(\cdot) : \mcl{C}_{\mcl{X}_{i,s}} \to (-\pi, \pi)^{n} 
        \text{ is Lipschitz on } \mcl{C}_{\mcl{X}_{i,s}} . 
    \end{equation} 
(See also \eqref{E:Lip.delta.k.on.Gi}.) Therefore, by \eqref{E:mclC.Xis.defn} and \eqref{E:Gi.is.open}, 
    \begin{equation}  \label{E:mclC.mclX.is.is.open}
      \mcl{C}_{\mcl{X}_{i,s}} \text{ is open.}
    \end{equation}. 

We have seen that if $(y,v) \in C[\mcl{X}_{i,s}]$, then the corresponding $\deli$ satisfies \eqref{E:condns.on.delta.i}. Hence, by \eqref{E:Exp(C[mclX])}, 
$Exp \bigl( C[\mcl{X}_{i,s}] \bigr) \subset \mcl{C}_{i,s}$. 

Conversely, let $x \in \mcl{C}_{i,s}$. We can write 
$x = \, \wrap \bigl( \deli(x) + \alpha_{2} 1_{n} \bigr)$, where $\alpha_{2} \in \RR$ and  
$\deli(x)$ satisfies \eqref{E:condns.on.delta.i}. Therefore, by lemma \ref{L:lambda.t.theta.uniqueness}, there exist unique solutions to 
\eqref{E:delta.i.defined.in.trms.of.t.lambda.theta} $\lambda = \lambda(\deli)$, 
$t = t(\deli)$, and $\theta = \theta(\deli)$ satisfying \eqref{E:Uis.params}. Hence, 
$x = \, \wrap \bigl( t e_{i} + \lambda (\theta - t e_{i}) + \alpha_{2} 1_{n} \bigr)$.
I.e., $x = Exp_{y}(\lambda (\theta - t e_{i}))$, where 
$y = \, \wrap(t e_{i} + \alpha_{2} 1_{n})$. Thus, $y \in \mcl{X}_{i,s}$ 
and $x \in C[y]$. In other words, 
$\mcl{C}_{i,s} \subset Exp \bigl( C[\mcl{X}_{i,s}] \bigr)$ and, hence, 
    \begin{equation*} 
      \mcl{C}_{i,s} = Exp \bigl( C[\mcl{X}_{i,s}] \bigr) .
    \end{equation*}

Now we develop an alternate description of $\delta \in \RR^{n}$ s.t.\ if $y \in \T$ then 
$(y, \delta) \in C_{\T}[y]$. \emph{Claim:} 
    \begin{multline}  \label{E:exchange.i.and.j.in.ineq}
     \text{If $j \neq i$ and equality holds at either side of } 
       -\frac{n-2}{n} s \deli_{i} \leq s \deli_{j} \leq s \deli_{i} \\
         \text{ then the same equality holds with $i$ and $j$ reversed.}
    \end{multline}
I.e., if either of the inequalities in \eqref{E:delta.(n-2)/n.ineq} is an equality, then the roles of $i$ and $j$ can be reversed. 
To prove this, first suppose 
$-\tfrac{n-2}{n} s \deli_{i} = s \deli_{j}$. We may assume $s = +1$. Then, by \eqref{E:delta.i.defn.2}, we have 
  \begin{align*}
    -\frac{n-2}{n} \left( \xi_{i} - \frac{1}{n-1} \sum_{k \neq i} \xi_{k} \right)
      &= \xi_{j} - \frac{1}{n-1} \sum_{k \neq i} \xi_{k} \\
       \text{ if and only if} -(n-2) \xi_{i} + \frac{n-2}{n-1} \sum_{k \neq i} \xi_{k} 
      &= n \xi_{j} - \frac{n}{n-1} \sum_{k \neq i} \xi_{k} .
  \end{align*}
Adding $\tfrac{n}{n-1} \sum_{k \neq i} \xi_{k}$ to both sides of the second of the preceding equations, with get 
  \begin{multline*}
    -n \xi_{i} + 2 \sum_{k = 1}^{n} \xi_{k} =
      -(n-2) \xi_{i} + 2 \sum_{k \neq i} \xi_{k} = 
        -(n-2) \xi_{i} + 2 \frac{n-1}{n-1} \sum_{k \neq i} \xi_{k} \\
          = n \xi_{j} .
  \end{multline*}
Subtracting $2 \xi_{j}$ from both sides yields
  \begin{multline*}
    -n \xi_{i} + 2 (n-1) \bar{\xi}^{j} = 
      -n \xi_{i} + 2 \sum_{k \neq j} \xi_{k} \\
        = (n-2) \xi_{j} . 
  \end{multline*}
Now subtract $(n-2) \bar{\xi}^{j}$ from both sides:
  \begin{multline*}
    -n \delta^{j}_{i} =
      -n (\xi_{i} - \bar{\xi}^{j}) \\
        = (n-2) \xi_{j} - (n-2) \bar{\xi}^{j}_{i} = (n-2) \delta^{j}_{j} . 
  \end{multline*}
I.e., $- \tfrac{n-2}{n} \delta^{j}_{j} = \delta^{j}_{i}$, as desired. 

Suppose $i \neq j$ but $\deli_{j} = \deli_{i}$. Then 
  \begin{equation*}
    \xi_{j} - \frac{1}{n-1} \sum_{k \neq i} \xi_{k} 
      = \xi_{i} - \frac{1}{n-1} \sum_{k \neq i} \xi_{k} .
  \end{equation*}
I.e., $\xi_{j} = \xi_{i}$. Recall \eqref{E:N.sub.n}. Hence, for any $\ell \in \NN_{n}$, 
  \begin{equation*}
    \delta^{\ell}_{i} = \xi_{i} - \bar{x}^{\ell} = \xi_{j} - \bar{x}^{\ell} = \delta^{\ell}_{j} .
  \end{equation*}
In particular this is true when $\ell = j$. We have proved the claim \eqref{E:exchange.i.and.j.in.ineq} and the general fact: 
  \begin{equation}  \label{E:equality.switcheroo}
    \text{If } \deli_{j} = \deli_{i}, \text{ then for any } \ell \in \NN_{n}, \; 
    \delta^{\ell}_{i} = \delta^{\ell}_{j} .
  \end{equation}
 
Thus, we have the following. Let $i, j \in \NN_{n}$ be distinct. Suppose $x \in \D$ has these properties: $\xi = \xi(x) \in \RR^{n}$ satisfies \eqref{E:xk.near.xi.i.bar} and also satisfies it with $i$ replaced by $j$, \eqref{E:delta.(n-2)/n.ineq} holds with one ``$\leq$'' replaced by ``$=$'', and $\Pf_{1i}$ is a closest lobe to $x$. Then, by lemma \ref{L:tween.-(n-2)/n.and.1}, $\Pf_{1j}$ is also a closest lobe to $x$. 

Given $\xi$, define a relation ``$\sim \, = \, \sim_{\xi}$'' on $\NN_{n}$ as follows. $j \sim i$ if and only if \emph{either} of the inequalities in \eqref{E:delta.(n-2)/n.ineq} is an equality. I.e., either 
$s \deli_{j} = -\frac{n-2}{n} s \deli_{i}$ \emph{or} $s \deli_{j} = s \deli_{i}$, where 
$s = sign(\deli_{i}$. This is obviously reflexive: $i \sim i$. By \eqref{E:exchange.i.and.j.in.ineq}, $\sim$ is symmetric: $i \sim j$ if and only if 
$j \sim i$. 

\emph{Claim:} $\sim$ is transitive. Let $i, j, \ell \in \NN_{n}$. Suppose 
$i \sim j$ and $j \sim \ell$. We show $i \sim \ell$. The only interesting case is $i,j,\ell$ distinct. 

Case 1: $\delta^{j}_{i} = \delta^{j}_{j}$ (so $i \sim j$) and 
$\delta^{\ell}_{j} = \delta^{\ell}_{\ell}$ (so $j \sim \ell$). By symmetry, 
$\delta^{j}_{\ell} = \delta^{j}_{j}$. Hence, $\delta^{j}_{i} = \delta^{j}_{\ell}$. By \eqref{E:equality.switcheroo} this means $\delta^{\ell}_{i} = \delta^{\ell}_{\ell}$. 
I.e., $i \sim \ell$. 

Case 2: $\delta^{j}_{i} = \delta^{j}_{j}$ (so $i \sim j$) and 
$\delta^{\ell}_{j} = -(n-2) \delta^{\ell}_{\ell}/n$ (so $j \sim \ell$). By \eqref{E:equality.switcheroo}, 
$\delta^{\ell}_{i} = \delta^{\ell}_{j}$. Hence, 
$\delta^{\ell}_{i} = -(n-2) \delta^{\ell}_{\ell}/n$. I.e., $i \sim \ell$. 

Case 3: $\delta^{j}_{i} = -(n-2) \delta^{j}_{j}/n$ ($i \sim j$) and 
$\delta^{\ell}_{j} = \delta^{\ell}_{\ell}$ ($j \sim \ell$). By \eqref{E:exchange.i.and.j.in.ineq}, 
$\delta^{j}_{\ell} = \delta^{j}_{j}$. Also by \eqref{E:exchange.i.and.j.in.ineq}, we have 
$\delta^{i}_{j} = -(n-2) \delta^{i}_{i}/n$. This is just Case 2 with $i$ and $\ell$ interchanged. Therefore, $\ell \sim i$. By symmetry, this implies $i \sim \ell$, as desired. 

Case 4: $\delta^{j}_{i} = -(n-2) \delta^{j}_{j}/n$ ($i \sim j$) and 
$\delta^{\ell}_{j} = -(n-2) \delta^{\ell}_{\ell}/n$ ($j \sim \ell$). By \eqref{E:exchange.i.and.j.in.ineq}, 
$\delta^{j}_{\ell} = -(n-2) \delta^{j}_{j}/n = \delta^{j}_{i}$. From \eqref{E:equality.switcheroo}, we then derive $\delta^{i}_{\ell} = \delta^{i}_{i}$ 
($\ell \sim i$). By symmetry $i \sim \ell$. This complete the proof of the claim that 
$\sim$ is transitive. 

Thus, $\sim$ is an equivalence relation. If $\sim$ is an equivalence relation on $\NN_{n}$, 
then $\NN_{n}/\sim$ denotes the set of equivalence classes induced by $\sim$. 

Suppose $\deli_{j} = -(n-2) \deli_{i}/n$. Then 
$\xi_{j} - \bar{\xi}^{i} = \deli_{j} = -(n-2) \deli_{i}/n = -(n-2)(\xi_{i} - \bar{\xi}^{i})/n$. Therefore, 
    \begin{equation*}
      \xi_{j} = - \frac{n-2}{n} \xi_{i} + 2 \frac{n-1}{n} \bar{\xi}^{i} . 
    \end{equation*}
Hence, 
    \begin{align}  \label{E:xij.bar.vs.xii.bar}
      \bar{\xi}^{j} = \frac{1}{n-1} \sum_{k \neq j} \xi_{k}
         &= \frac{1}{n-1} \left( \sum_{k \neq i} \xi_{k} - \xi_{j} + \xi_{i} \right) \notag \\
         &= \frac{1}{n-1} \left( \sum_{k \neq i} \xi_{k} - (\xi_{j} - \bar{\xi}^{i}) 
           + (\xi_{i} - \bar{\xi}^{i}) \right) \notag \\
         &= \frac{1}{n-1} \left( \sum_{k \neq i} \xi_{k} - \deli_{j} + \deli_{i} \right) \\
         &= \frac{1}{n-1} \left( \sum_{k \neq i} \xi_{k} + \frac{n-2}{n} \deli_{i} 
           + \deli_{i} \right) \notag \\
         &= \frac{1}{n-1} \left( \sum_{k \neq i} \xi_{k} + 2 \frac{n-1}{n} \deli_{i} \right) \notag \\
         &= \bar{\xi}^{i} + \frac{2}{n} \deli_{i} \notag .
    \end{align}
It follows that 
    \begin{equation}  \label{E:delta.j.=.-delta.i-2.delta.ii/n}
      \delta^{j} = \deli - \tfrac{2}{n} \deli_{i} 1_{n} .
    \end{equation} 
In particular, 
    \begin{equation}  \label{E:delta.jj.=.-delta.ii}
          \delta^{j}_{j} 
          = \deli_{j} - \frac{2}{n} \deli_{i}
            = -\frac{n-2}{n} \deli_{i} - \frac{2}{n} \deli_{i} = - \deli_{i} .
    \end{equation}
We conclude that 
    \begin{equation} \label{E:i.sim.j.means.deltajj=+-deltaii}
      \text{if $i \sim j$ then } s_{j} \delta^{j}_{j} = s_{i} \deli_{i} , 
        \text{ where } s_{i} = sign(\deli_{i}) \text{ and }
          s_{j} = sign(\delta^{j}_{j}) .
    \end{equation}
By the preceding and \eqref{E:xij.bar.vs.xii.bar},
    \begin{equation}  \label{E:xi.bar.i+delta.ii}
      \bar{\xi}^{j} + \frac{1}{n} \delta^{j}_{j} = \bar{\xi}^{i} + \frac{1}{n} \deli_{i} .
    \end{equation}

The discrepancy \eqref{E:xij.bar.vs.xii.bar} can cause problems as $t \to 0$. To solve that problem, recall that, by \eqref{E:condns.on.delta.i}, 
$\sum_{k \neq i} \deli_{k} = 0$. Therefore, 
    \begin{equation}  \label{E:mean.of.delta.i}
      \text{mean of } \deli = \frac{1}{n} \sum_{k = 1}^{n} \deli_{k} 
        = \frac{1}{n} \deli_{i} + \frac{1}{n} \sum_{k \neq i} \deli_{k} 
          = \frac{1}{n} \deli_{i} .
    \end{equation}

\eqref{E:Exp(C[mclX])} gets revised as follows. Recall \eqref{E:mclC.Xis.defn} and \eqref{E:C[y].y.in.Xis}. Let $(y,v) \in C[\mcl{X}_{i,s}]$. For reasons to be soon made apparent, write $y_{\pi} = y$. We may write 
$y_{\pi} = \, \wrap( \alpha_{1} e_{i} + \alpha_{2} 1_{n} ) \in \mcl{X}_{i,s}$. Then, for  for some $\theta$, and $\lambda$ satisfying \eqref{E:Uis.params}, we have 
$v = \lambda (\theta - \alpha_{1} e_{i})$. Let $t = \alpha_{1}$ so $|t| \in (0,\pi)$. Let 
$\deli = \alpha_{1} e_{i} + v = t e_{i} + \lambda (\theta - \alpha_{1} e_{i})$. Then, as remarked before \eqref{E:mclC.Xis.defn}, $\deli$ satisfies \eqref{E:condns.on.delta.i}. By \eqref{E:when.delta.is.nice}, that means $\deli \in U_{i,s}$. 
Now let $\gamma_{2}^{i} := \alpha_{2}  - (\pi - |\alpha_{1}|) (n \pi)^{-1} \deli_{i}$, 
so $\gamma_{2}^{i} \in \mcl{X}_{i,s}$ and 
$y_{t} := \, \wrap( \alpha_{1} e_{i} + \gamma_{2}^{i} 1_{n} )$.  
Then, 
    \begin{multline}  \label{E:Exp(C[mclX]).revised}
      Exp \bigl( (y_{\pi},v) \bigr)  
        = \, \wrap \Bigl[ \alpha_{1} e_{i} + \alpha_{2} 1_{n} 
          + \lambda (\theta - \alpha_{1} e_{i}) \Bigr] \\
        = \, \wrap \Bigl( \bigl[ \alpha_{1} e_{i} + \gamma_{2}^{i} 1_{n} \bigr]
            + \bigl[ \lambda (\theta - \alpha_{1} e_{i}) 
            - (\pi - |\alpha_{1}|) (n \pi)^{-1} \deli_{i} 1_{n} \bigr] \Bigr) \\
            = Exp \Bigl[ \bigl( y_{t}, \lambda (\theta - \alpha_{1} e_{i}) 
            - \bigl( \pi - |\alpha_{1}| \bigr) (n \pi)^{-1} \deli_{i} 1_{n} \Bigr] \\
              \text{ where }
                y_{t} = \, \wrap( \alpha_{1} e_{i} + \gamma_{2}^{i} 1_{n} ) 
                     \in \mcl{X}_{i,s} .
    \end{multline}

Straying beyond $C[\mcl{X}_{i,s}]$, if $|t| = \alpha_{1} = \pi$ then 
$\gamma_{2}^{i} := \alpha_{2}$. On the other hand, suppose $t = \alpha_{1} = 0$. Then $\gamma_{2}^{i} := \alpha_{2}  - n^{-1} \deli_{i}$. This occurs if and only if for some $j \in \NN_{n}$ we have $j \sim i$. Specifically, first assume $\deli_{j} = \deli_{i}$, 
so $\xi_{j} = \xi_{i}$. Then, obviously $\bar{\xi}^{j} = \bar{\xi}^{j}$ and 
$\delta^{j}_{j} = \deli_{i}$ and $y_{t} = \, \wrap( \alpha_{1} e_{i} + \gamma_{2}^{i} 1_{n} 
+ n^{-1} \deli_{i} 1_{n} ) \in \mcl{X}_{i,s}$ in \eqref{E:Exp(C[mclX]).revised} is unchanged when $i$ is replaced by $j$. 

Now suppose $\deli_{j} = -(n-2) \deli_{i} /n$. Then, by \eqref{E:delta.j.=.-delta.i-2.delta.ii/n} and \eqref{E:delta.jj.=.-delta.ii}, 
    \begin{align*}
      \delta^{j} + \gamma_{2}^{j} 1_{n} 
        &= \deli - \tfrac{2}{n} \deli_{i} 1_{n} 
          + (\alpha_{2}  - n^{-1} \delta^{j}_{j}) 1_{n} \\
        &= \deli - \tfrac{2}{n} \deli_{i} 1_{n} 
            + \left( \alpha_{2}  + \tfrac{1}{n} \deli_{i} \right) 1_{n} \\
        &= \deli + \left( \alpha_{2}  - \tfrac{1}{n} \deli_{i} \right) 1_{n}
            = \deli + \gamma_{2}^{i} 1_{n} .
    \end{align*}
This shows that in \eqref{E:Exp(C[mclX]).revised}, $i$ can be replaced by $j \sim i$. Moreover, $y_{0} = \, \wrap( \alpha_{1} e_{i} + \alpha_{2} 1_{n} )$ is independent of $i$.  

Notice that, by \eqref{E:mean.of.delta.i}, with $t = \alpha_{1} = 0$, we have
    \begin{equation*}
      \text{mean of } (\deli + \gamma_{2}^{i} 1_{n}) 
        = \text{mean of } (\deli + \alpha_{2} 1_{n}  - n^{-1} \deli_{i} 1_{n}) =  
          \text{mean of } (\alpha_{2} 1_{n}) = \alpha_{2} , 
    \end{equation*}
again independent of $i$. 

The intuition behind introducing the bias 
$\gamma_{2}^{i} := \alpha_{2}  - (\pi - |\alpha_{1}|) (n \pi)^{-1} \deli_{i}$ has to do with the location of $\xi$ in $\Pf_{1}$. If some $\deli_{i}$ is big, say nearly $\pi$ (so, by \eqref{E:generic.eqn.for.t}, $|t|$ is nearly $\pi$), that suggests that the other coordinates correspond to a point of $\T$ but $\deli_{i}$ does not, so $\xi$ should be thought of as approximated by some $\alpha_{1} e_{i} + \alpha_{2} 1_{n}$. One the other hand, if $\deli \sim \delta^{j}$ (so $t = 0$) it is hard visualize $\deli_{i}$ as an outlier. Therefore, 
$\xi$ is approximated by some $\alpha_{2} 1_{n}$. For $|t| \in (0,\pi)$, 
$\alpha_{1} e_{i} + \gamma_{2}^{i} 1_{n}+ n^{-1} \deli_{i} 1_{n}$ interpolates between the two cases.

Going in the other direction, suppose we are given 
$\xi \in \mcl{G}_{i}$. (See \eqref{E:Gi.defn}.) Then, $x = \, \wrap(\xi)$, where $\xi$ has the properties ascribed to it in \eqref{E:xi.parallels.x.on.Gi}. Then, by \eqref{E:delta.i.from.xi}, 
$\xi = \deli + \bar{\xi}^{i}$. Write $\deli$ as in \eqref{E:delta.i.defined.in.trms.of.t.lambda.theta}. Then we have 
    \begin{equation*}
      \alpha_{1} = t \text{ and } 
        \alpha_{2} =  \bar{\xi}^{i} - \frac{\pi - |t| }{n \pi} \deli_{i} .
    \end{equation*} 
So, by \eqref{E:xi.bar.i+delta.ii}, $\alpha_{1}$ and $\alpha_{2}$ depend on $i$ 
only if $t \neq 0$. 

\emph{Claim:} 
    \begin{multline}  \label{E:ineqs.with.i,j}
      \text{Let $s = sign(\deli_{i})$ and suppose }  
        -\frac{n-2}{n} s \deli_{i} \leq s \deli_{k} \leq s \deli_{i} 
         \text{ for every } k \neq i . \\
           \text{ Suppose for some $j \neq i$ we have } j \sim i. \\
             \text{ Then the above inequalities also hold if $i$ is replaced by $j$, including }
               s = sign(\delta^{j}_{j}) .
    \end{multline}
(Trivially, the inequalities hold with $k = i$.) Suppose $j \neq i$ and the hypotheses in the preceding hold including $j \sim i$. WLOG $sign(\deli_{i}) = +1$. (By \eqref{E:sign.function}, this includes the case $\deli_{i} = 0$.) Let $s_{j} := sign(\delta^{j}_{j})$. We show
    \begin{equation}  \label{E:ineqs.with.j}
      -\frac{n-2}{n} s_{j} \delta^{j}_{j} \leq  \delta^{j}_{k} \leq s_{j} \delta^{j}_{j} 
             \text{ for every } k \neq j .
    \end{equation}

First, suppose $\deli_{j} = -(n-2) \deli_{i}/n$. Then by symmetry of $\sim$, more specifically \eqref{E:exchange.i.and.j.in.ineq}, we have $\delta^{j}_{i} = -(n-2) \delta^{j}_{j}/n$. Then,
$s_{j} \delta^{j}_{i} = -(n-2) s_{j} \delta^{j}_{j}/n$. By assumption, for every $k \in \NN_{n}$, 
$\xi_{k} - \bar{\xi}^{i} = \deli_{k} \leq \deli_{i} = \xi_{i} - \bar{\xi}^{i}$. Thus, 
$\xi_{k} \leq \xi_{i}$. With $k = j$ we get 
    \begin{equation}  \label{E:delta.jj.leq.0}
      \delta^{j}_{j} = \xi_{j} - \bar{x}^{j} \leq \xi_{i} - \bar{x}^{j} = \delta^{j}_{i} 
        = -\frac{n-2}{n} \delta^{i}_{i} \leq 0. 
    \end{equation} 
Thus, $s_{j} = -1$. Therefore, for any $k \in \NN_{n}$, $s_{j} \xi_{k} \geq s_{j} \xi_{i}$. I.e. 
    \begin{equation*}
      s_{j} \delta^{j}_{k} = s_{j} (\xi_{k} - \bar{\xi}^{j}) \geq s_{j} (\xi_{i} - \bar{\xi}^{j}) 
        = s_{j} \delta^{j}_{i} = -\frac{n-2}{n} s_{j} \delta^{j}_{j} .
    \end{equation*}

Next, still with $\deli_{j} = -(n-2) \deli_{i}/n$, we show 
$s_{j} \delta^{j}_{k} \leq s_{j} \delta^{j}_{j}$. Let $k \in \NN_{n}$. By assumption,
    \begin{equation*}
      \xi_{k} - \bar{\xi}^{i} = \deli_{k} \geq -\frac{n-2}{n} s \deli_{i} = \deli_{j} 
       = \xi_{j} - \bar{\xi}^{i} . 
    \end{equation*}
Thus, $\xi_{k} \geq \xi_{j}$. Recall that, by \eqref{E:delta.jj.leq.0}, $s_{j} = -1$. Then, 
    \begin{equation*}
      s_{j} \delta^{j}_{k} = s_{j} (\xi_{k} - \bar{\xi}^{j}) \leq s_{j} (\xi_{j} - \bar{\xi}^{j})
        = s_{j} \delta^{j}_{j} ,
    \end{equation*}
as desired. 

Next, suppose $\deli_{j} = \deli_{i}$, so $s_{j} := sign(\delta^{j}_{j})$, which we are assuming is $+1$. Then $\xi_{j} - \bar{\xi}^{i} = \xi_{i} - \bar{\xi}^{i}$ so $\xi_{j}= \xi_{i}$. \emph{This means $\bar{\xi}^{j} = \bar{\xi}^{i}$.} Hence, $\delta^{j}_{k} = \deli_{k}$ for every 
$k \in \NN_{n}$. Therefore, by assumption on $i$, 
    \begin{equation*}
       \delta^{j}_{k} = \deli_{k} \leq \deli_{i} = \delta^{j}_{j} \text{ and }
         \delta^{j}_{k} = \deli_{k} \geq -\frac{n-2}{n} \deli_{i} 
           = -\frac{n-2}{n} \delta^{j}_{j} \text{ for every } k \neq j .
    \end{equation*}
as desired. This completes the proof of claim \eqref{E:ineqs.with.i,j}. 

Let $y \in \T$ and $\delta = (\delta_{1}, \ldots, \delta_{n}) \in \RR^{n}$. 
Then $(y,\delta) \in C_{\T}$ if and only if there exists an $\sim \,$-equivalence class 
$I \subset \NN_{n}$ \emph{consisting of more than one element} and s.t.\ if $i \in I$ then 
$-\frac{n-2}{n} s \deli_{i} \leq s \deli_{k} \leq s \deli_{i}$, where $s = sign(\deli_{i})$, for every $k \neq i$. (By \eqref{E:ineqs.with.i,j}, the same thing is true if $i$ is replaced by any $j \in I$.) 

\section{Cone bundles over $\mcl{Y}_{i}$ and $\T$} \label{S:C[mclY.i,T]}
We still have to attach cone fibers to the points of $\Y_{i}$ 
and $\T$.

If $A \subset \RR^{n}$ define $A + \T := \{a + x : a \in A, x \in \T \} \subset T \D$. We have shown that $U_{i,s} + \T$ is an open neighborhood of $\mcl{X}_{i,s}$ in $T \D \restriction_{\mcl{X}_{i,s}}$. ($T \D \restriction_{\mcl{X}_{i,s}}$ is the restriction 
of $T \D$ to $\mcl{X}_{i,s}$.)  

Let $\Theta_{\mcl{Y}_{i}}$ be the set 
of $\theta = (\theta_{1}, \ldots, \theta_{n}) \in [-\pi, \pi]^{n}$ with the properties
    \begin{align}  \label{E:big.Theta.Yi.requirements.4}
          \theta \cdot f_{i} &= 0, \notag \\ 
          - \pi &\leq \theta_{j} \leq \pi \quad (j \neq i), \\
          \theta_{i} &= \pi , \text{ and } \notag \\
          \text{ there exists } j &\neq i \text{ s.t.\ } |\theta_{j}| = \pi \notag .
    \end{align} 
Compare this to \eqref{E:big.Theta.Xit.requirements}. 

Let $y \in \mcl{Y}_{i}$ and define 
     \begin{equation}  \label{E:C[y].y.in.Yis} 
      C_{\mcl{Y}_{i}}[y] := \Bigl\{ \bigl( y, \lambda (\theta - \pi e_{i}) \bigr) \in T_{y} \D : 
        \lambda \in [0, r/\pi), \, \theta \in \Theta_{\mcl{Y}_{i}} \Bigr\} . 
    \end{equation} 
Compare this to \eqref{E:C[y].y.in.Xis}. Notice that if $(y,v) \in C_{\mcl{Y}_{i}}[y]$ 
then $v \perp e_{i}$.

Recall \eqref{E:Lip.delta.k.on.Gi}. Let $\mcl{C}_{\mcl{Y}_{i}} := \mcl{C}_{i,s} \subset \D$ be the set of $x \in \D$ s.t.:
    \begin{align}  \label{E:C.Yi.defn}
      x &\in \mcl{G}_{i} , \notag \\
      \deli_{i} &:= \pi , \\ 
      \deli_{j} &\in (-r,r) . \notag 
    \end{align}
For $x \in \mcl{C}_{i,s}$ let 
$\mcl{F}(x) := \bigl\{ z \in \mcl{C}_{i,s} : \angle_{s}(\bar{x}^{i}, z_{i}) < \pi/4 \bigr\}$. Then, by \eqref{E:directed.angle.is.Lip} and \eqref{E:locally.Lip.xbari.on.Gi},  
there exists $L < \infty$ s.t.\ $\bigl| \deli(z) - \deli(x) \bigr| < L \, \rho(z,x)$ for every 
$z \in \mcl{E}(x) \cap \mcl{F}(x)$. 
In general, for $x,z \in \mcl{C}_{i,s}$, we have 
$\bigl| \deli(z) - \deli(x) \bigr| \leq \sqrt{n} \pi$. Moreover, if 
$z \in \mcl{C}_{i,s} \setminus \bigl[ \mcl{E}(x) \cap \mcl{F}(x) \bigr]$ then 
$\rho(z,x) \geq \pi/4$. So if 
$z \in \mcl{C}_{i,s} \setminus \bigl[ \mcl{E}(x) \cap \mcl{F}(x) \bigr]$ then
$\bigl| \deli(z) - \deli(x) \bigr| \leq \bigl[ \sqrt{n} \pi/(\pi/4) \bigr] \rho(z,x)$. 
Let $L' := \max(L, 4 \sqrt{n})$. It follows that, if 
$x,z \in \mcl{C}_{i,s}$ then $\bigl| \deli(z) - \deli(x) \bigr| \leq L' \, \rho(z,x)$. This proves, 
    \begin{equation}  \label{E:Lip.delta.on.Gi}  
      \text{The map } \deli(\cdot) : \mcl{C}_{\mcl{Y}_{i}} \to (-\pi, \pi)^{n} 
        \text{ is Lipschitz on } \mcl{C}_{\mcl{Y}_{i}} . 
    \end{equation} 

It is easy to see that $\mcl{C}_{\mcl{Y}_{i}}$ is open and the $\mcl{Y}_{i}$-analogue of \eqref{E:Exp(C[mclX])} holds?

Recall \eqref{E:Gi.defn} and \eqref{E:mclC.Xis.defn}. Define $\mcl{C}_{i}$ to be the set of all $x \in \D$ s.t.\ 
    \begin{align}  \label{E:mclCi.defn}
      x &\in \mcl{G}_{i} , \notag \\
      \deli_{i} &:= s\angle_{s}(x_{i}, \bar{x}^{i}) \in (-\pi, 0) \cup (0, \pi] , \\
      \deli_{j} &:= s\angle_{s}(x_{j}, \bar{x}^{i}) \in (-r,r) \; \text{ if } j \neq i , 
        \text{ and } \notag \\
      -\frac{n-2}{n} s \deli_{i} &< s \deli_{j}(x) < s \deli_{i} \; \text{ if } j \neq i . \notag
    \end{align}
By lemma \ref{L:tween.-(n-2)/n.and.1} we have
    \begin{equation}  \label{E:Pf1i.is.closest.lobe.to.mclCi}
      \text{If $x \in \mcl{C}_{i}$ then $\Pf_{1i}$ is the unique closest lobe 
        of $\Pf_{1}$ to $x$.}
    \end{equation}
It follows that
    \begin{equation}  \label{E:mclCi's.are.disjoint}
      \text{The sets $\mcl{C}_{i}$, $i = 1, \ldots, n$, are disjoint.} 
    \end{equation} 

\emph{Claim:} 
    \begin{equation}  \label{E:mclCi.is.open}
      \mcl{C}_{i} \text{ is open.} 
    \end{equation}
By \eqref{E:mclC.mclX.is.is.open}, it suffices to show that if $x \in \mcl{C}_{i}$ with $\deli(x) = \angle(x, \bar{x}^{i})) = \pi$ then $x$ has a neighborhood in $\mcl{C}_{i}$. Let $x \in \mcl{C}_{i}$ with 
$\deli(x) = \angle(x, \bar{x}^{i})) = \pi$. Then $x \in \mcl{G}_{i}$. By \eqref{E:Gi.is.open}, 
$\mcl{G}_{i}$ is open. 
If $x' \in \mcl{G}_{i}$ is sufficiently close to $x$ then, by \eqref{E:directed.angle.is.Lip} and \eqref{E:Lip.delta.k.on.Gi}, $x' \in \mcl{C}_{i}$. 

Next, we construct a cone bundle over $\T$. 
Let $I \subset \NN_{n}$ with $|I| > 1$, where $|I|$ is the cardinality of $I$. Let $\Theta_{\T,I}$ be the set of 
$\theta = (\theta_{1}, \ldots, \theta_{n}) \in [-\pi, \pi]^{n}$ with the following properties. 
     \begin{align}  \label{E:big.Theta.T.requirements}
     \theta \cdot 1_{n} = 0, \; 
     I \text{ is an equivalence class of } \sim_{\theta} , \\
       \text{ and if $j \notin I$ then }  \notag  
          - \frac{n-2}{n} \pi < s \theta_{j} < \pi \quad (j \neq i), \; 
    \end{align} 

This is just $\Theta_{i,t}$ (\eqref{E:big.Theta.Xit.requirements}) with $|t| \leq r$ and 
$s = sign(t)$. In fact, for $\T$ we take 
$t = 0$ (so $s = \pm 1$ is arbitrary) and 
The $\T$-analogue of \eqref{E:delta.i.defined.in.trms.of.t.lambda.theta} is 
    \begin{equation*}  \label{E:T.style.delta}
      \delta := \lambda \theta ,
        \quad \lambda \in [0, r/\pi), \, \theta \in \Theta_{\T} .
    \end{equation*} 
Thus, $\delta$ is independent of $i$.

Recall \eqref{E:big.Theta.T.requirements}. Now, for $y \in T$ define  
    \begin{equation}  \label{E:C[y].y.in.T}
      C_{\T}[y] := \Bigl\{ ( y, \lambda \theta ) \in T_{y} \D : 
        \lambda \in [0, r/\pi), \, \theta \in \Theta_{\T} \Bigr\} .
    \end{equation}

As in the case of $y \in \mcl{X}_{i,s}$, we have that $C_{\mcl{Y}_{i}}[y] = \msf{CL}$ (for $y \in \mcl{Y}_{i}$) and $C_{\T}[y] = \msf{CL}$ (for $y \in \T$) where $\msf{L}$ satisfies \eqref{E:Xis.cone.has.tame.link}. Hence, \textbf{property \ref{I:L.A.tameness} of definition \ref{D:fibering.by.cones}} is satisfied in general. 

Comparing \eqref{E:big.Theta.Xit.requirements}, \eqref{E:big.Theta.Yi.requirements.4} and \eqref{E:big.Theta.T.requirements}. We see they are all are just special cases of the set, call it $\Theta_{i}$, of $\deli \in \RR^{n}$ satisfying \eqref{E:big.Theta.Xit.requirements} but with $t \in (-\pi, \pi]$ and $s = sign(t)$. (Note that 
$\Theta_{\T} = \bigcap_{i=1}^{n} \Theta_{i}$.)

\chapter*{Acknowledgments}  \label{Chptr:acknowledgments}
Permission to use the data in figure \ref{F:HeightFen} was given by John Mann of the New York State Psychiatric Institute at Columbia University. My memory of this is hazy, but I seem to recall that I got the idea of looking at the dimension of singular sets from Doug Ravenel (personal communication). I seem to also recall that Doug conjectured that the directional mean would achieve the bound on the dimension of singular sets (\eqref{E:codim.Sa.=.nvar+1}). 

I discussed some of the material in chapter \ref{Chptr:sings.in.plane.fit} with Martin Bendersky. Fred Cohen told me about Eckmann \emph{et al} \cite{bEtGpjH62.genMeans}, which is relevant to chapter \ref{Chptr:spherical.location}. John Harper recommended the book Hurewicz and Wallman \cite{wHhW48.DimThy}, which is where I got \eqref{E:cohom.and.codim}. I seem to recall briefly communicating with Markus Banagl concerning issues involved in definition \ref{D:fibering.by.cones}. Steven Ferry provided me with methods for showing that \eqref{E:nontriv.r-dim.homol} holds in the context of section \ref{SS:gnrl.lwr.bnd.plane.fit} (personal communication). 

\newcommand{\etalchar}[1]{$^{#1}$}
\providecommand{\bysame}{\leavevmode\hbox to3em{\hrulefill}\thinspace}
\providecommand{\MR}{\relax\ifhmode\unskip\space\fi MR }
\providecommand{\MRhref}[2]{%
  \href{http://www.ams.org/mathscinet-getitem?mr=#1}{#2}
}
\providecommand{\href}[2]{#2}

\end{document}